\documentclass[10pt]{surv-l}
\usepackage{amssymb,amsmath,exscale,latexsym,amsthm,graphics,overpic,hyphenat,mathtools,caption,subcaption}
\usepackage{color}
\usepackage{enumitem}
\usepackage{epsfig}    
\usepackage{epic}      
\usepackage{eepic} 
\usepackage[matrix,arrow,curve,cmtip]{xy} 
\usepackage{letltxmacro}
\usepackage{ifthen} 
 
\usepackage{xcolor}
\usepackage{hyperref}
   
\makeatletter
\renewcommand{\l@figure}{\@tocline{0}{3pt plus2pt}{0pt}{3em}{}}
\makeatother

\makeatletter
\newtheorem*{rep@theorem}{\rep@title}
\newcommand{\newreptheorem}[2]{%
\newenvironment{rep#1}[1]{%
 \def\rep@title{#2 \ref{##1}}%
 \begin{rep@theorem}}%
 {\end{rep@theorem}}}
\newtheorem*{rep@prop}{\rep@title}
\newcommand{\newrepprop}[2]{%
\newenvironment{rep#1}[1]{%
 \def\rep@title{#2 \ref{##1}}%
 \begin{rep@prop}}%
 {\end{rep@prop}}}
\newtheorem*{rep@cor}{\rep@title}
\newcommand{\newrepcor}[2]{%
\newenvironment{rep#1}[1]{%
 \def\rep@title{#2 \ref{##1}}%
 \begin{rep@cor}}%
 {\end{rep@cor}}}

\makeatother

\setenumerate{itemsep=3pt,topsep=3pt}
\setenumerate[1]{label=\upshape(\roman*)}

\theoremstyle{plain} 
        \newtheorem{theorem}{Theorem}[chapter]
        \newtheorem*{theorem*}{Theorem}
        \newreptheorem{theorem}{Theorem}
        \newtheorem*{conj*}{Conjecture} 
        \newtheorem{lemma}[theorem]{Lemma}
        \newtheorem{prop}[theorem]{Proposition}
        \newrepprop{prop}{Proposition}
        \newtheorem{cor}[theorem]{Corollary}
         \newrepcor{cor}{Corollary}

        \newtheorem{prob}{Problem}
        \newtheorem*{quasi_u_problem*}{The Quasisymmetric Uniformization Problem}

\theoremstyle{definition}
        \newtheorem{definition}[theorem]{Definition}
        \newtheorem{rem}[theorem]{Remark}
         \newtheorem{ex}[theorem]{Example}
          
\theoremstyle{remark}

\numberwithin{equation}{chapter}
\numberwithin{section}{chapter}
\numberwithin{figure}{chapter}

\providecommand{\defn}[1]{\emph{#1}}

\newcounter{mylistnum}

\newcommand{\dist}{\operatorname{dist}}

\newcommand{\length}{\operatorname{length}}

\newcommand{\diam}  {\operatorname{diam}}

\newcommand{\inte}  {\operatorname{int}}

\newcommand{\id} {\operatorname{id}}

\newcommand{\lcm} {\operatorname{lcm}}
\newcommand{\acts}{\curvearrowright}
\newcommand{\R}{\mathbb{R}}  
\newcommand{\B}{\mathbb{B}}    
\newcommand{\C}{\mathbb{C}}      
\newcommand{\Q}{\mathbb{Q}}  
\newcommand{\N}{\mathbb{N}}      
\newcommand{\Z}{\mathbb{Z}}      
\newcommand{\T}{\mathbb{T}}      
\newcommand{\Sph}{\mathbb{S}}  

\newcommand{\Halb}{\mathbb{H}}   
  
\newcommand{\Om}{\Omega} 
\newcommand{\CDach}{\widehat{\mathbb{C}}}
\newcommand{\Cdach}{\widehat{\mathbb{C}}}
\newcommand{\D}{\mathbb{D}}      

\newcommand{\Aut}{\operatorname{Aut}}
\newcommand{\Isom}{\operatorname{Isom}}
\newcommand{\imag}{\operatorname{Im}}
\newcommand{\real}{\operatorname{Re}}
\newcommand{\leb}{\mathcal{L}}
\providecommand{\abs}[1]{\lvert#1\rvert}

\providecommand{\norm}[1]{\lVert#1\rVert}

\DeclarePairedDelimiter\ceil{\lceil}{\rceil}
\DeclarePairedDelimiter\floor{\lfloor}{\rfloor}
\renewcommand{\:}{\colon}
\newcommand{\ra}{\rightarrow}
\newcommand{\sub}{\subset}
\newcommand{\eps}{\epsilon}
\newcommand{\de}{\delta}
\newcommand{\ga}{\gamma}
\newcommand{\om}{\omega}

\newcommand{\iu}{\textbf{\textit{i}}}
\newcommand{\Sp}{S^2}

\newcommand{\crit}{\operatorname{crit}}
\newcommand{\post}{\operatorname{post}}

\newcommand{\mesh}{\operatorname{mesh}}
\newcommand{\LLC}{\operatorname{LLC}}

\newcommand{\CC}{\mathcal{C}}

\newcommand{\OC}{\mathcal{O}}

 \newcommand{\DD}{\mathcal{D}}
  \newcommand{\G}{\mathcal{G}}
\newcommand\geo{\partial_\infty}

\newcommand{\X} {\mathbf{X}}

\newcommand{\E} {\mathbf{E}}

\newcommand{\V} {\mathbf{V}}

\newcommand{\XOw} {X^0_{{\tt w}}}
\newcommand{\XOb} {X^0_{{\tt b}}}

\newcommand{\wt}{{\tt w}}
\newcommand{\bt}{{\tt b}}

\newcommand{\supsim}{\sqsupset}

\newcommand{\Gtr}{G_{\text{tr}}}

\newcommand{\Sim}{{\sim}}

\newcommand{\img}{\operatorname{img}}

\newcommand{\supp}{\operatorname{supp}}

\def\mate{\hskip 1pt \bot \hskip -5pt \bot \hskip 1pt}

\newboolean{nofigures}
\newboolean{singlechapter}

\setboolean{nofigures}{false} 

\setboolean{singlechapter}{false}

\makeindex

\begin{document}

\date{today}

\frontmatter
%

\title{Expanding Thurston Maps}

\author{Mario Bonk}
\address{Department of Mathematics, University of California, 
Los Angeles, CA 90095, USA}
\email{mbonk@math.ucla.edu}
\author{Daniel Meyer}
\address{Department of Mathematical Sciences, University of Liverpool
Mathematical Sciences Building,  Liverpool L69 7ZL,  United Kingdom}

\email{dmeyermail@gmail.com}


\date{April 24, 2016}


  \maketitle
  \tableofcontents

  \listoffigures


%


\chapter*{Preface}
 This book is the result of an intended research paper that grew out of control. 
 A preprint containing 
a substantial part of our investigations was already published on
arXiv in 2010. 
To make its content more accessible, 
we decided to include some additional material. 
These additions more than doubled 
the size of this work as compared with 
the 2010 version  and caused a long delay in its completion.


More than fifteen years ago we became 
both
interested in some basic problems on quasisymmetric
parametrization of $2$-spheres. 
This is related to the dynamics of rational maps---an observation we believe was first made by Rick Kenyon.
During our 
time at the University of Michigan  we  decided to  join forces and to investigate this connection systematically. 

We realized that  for the relevant  rational maps an explicit analytic expression is not so important, but rather a geometric-combinatorial description. 
As this became our preferred
way of looking  at these objects,  it was a 
natural step to consider a more general class of maps that are not necessarily  holomorphic. 
The relevant properties can be condensed into the notion of an {\em expanding Thurston map} which is the topic of this book. We will discuss the underlying ideas more thoroughly in the introduction (Chapter 1).

Part of this work 
overlaps
with studies by other researchers,
notably Ha\"\i ssinsky-Pilgrim \cite{HP}, and Cannon-Floyd-Parry
\cite{CFP07}. We would like to clarify some of the interrelations of our investigations with these works.  Theorem~\ref{thm:main} (in the body of the text)
was announced by the first author during an Invited Address at
the AMS Meeting at Athens, Ohio, in March 2004, where he gave a
short outline of the proof. After the talk he was informed by
Bill Floyd and Walter Parry that related results had been
independently obtained by Cannon-Floyd-Parry (which later
appeared as \cite{CFP07}).

Theorem~\ref{thm:S2vsf}~\ref{item:S2qsphere}  
was previously published 
by Ha\"\i
ssinsky-Pilgrim as part of a more general statement
\cite[Theorem~4.2.11]{HP}. Special cases go back to work by the
second author \cite{Me02} and unpublished joint work by
Bruce~Kleiner and the first author. The current, more general
version 
emerged after a visit of the first author
at the University of Indiana at Bloomington in February 2003.

During this visit the first author explained to Kevin Pilgrim concepts of
quasiconformal geometry and his joint work with
Bruce Kleiner on Cannon's conjecture in geometric group theory.
Kevin Pilgrim in turn pointed out Theorem~\ref{thm:exppromequiv}
and the ideas for its proof to the first author. After this
visit versions of Theorem~\ref{thm:S2vsf}~\ref{item:S2qsphere}  with an outline
for the proof were found independently by Kevin Pilgrim and the
first author. A proof of Theorem~\ref{thm:S2vsf}~\ref{item:S2qsphere}  was discovered soon
afterwards by the authors using ideas from \cite{Me02} (see
\cite{Me08} for an argument along similar lines) in combination
with Theorem~\ref{thm:main}.

We are indebted to many people. Conversations with Bruce Kleiner, Peter Ha\"{\i}ssinsky,
and Kevin Pilgrim have been especially fruitful.  We would also
like to thank Jim Cannon, Bill  Floyd, Lukas Geyer, 
Misha Hlushchanka, Zhiqiang
Li, Dimitrios Ntalampekos, Walter Parry,  Juan Souto, Dennis Sullivan, 
and Mike Zieve  for various useful comments. 
Two anonymous referees provided us with 
valuable feedback. Their considerable efforts were very much
appreciated. 
 
Qian Yin was so kind to let us incorporate parts of her thesis.
We are grateful to Jana Kleineberg for her careful proofreading
and her help with some of the pictures.
We are also happy to acknowledge the patient support of our   editors from the American Mathematical Society, Ed Dunne and  Ina Mette. 

Over the years we received funding from various sources.
Mario Bonk was partially supported by NSF grants DMS 0244421, DMS 0456940,  
DMS 0652915,  DMS 1058283, DMS 1058772,  DMS 1162471, and DMS 1506099.
Daniel Meyer  was partially supported by an NSF
  postdoctoral fellowship, the Deutsche Forschungsgemeinschaft (DFG-ME 4188/1-1), the
  Academy of Finland, projects SA-134757 and SA-118634, and the
  Centre of Excellence in Analysis and Dynamics Research, project No. 271983.

\medskip
\hfill Los Angeles and Liverpool, March, 2017

\ifthenelse{\boolean{singlechapter}}{
%




\newpage
\chapter*{Notation}\label{sec:not} 

We summarize some of the most important notation used in this book 
for easy reference. 

When an object $A$ is defined to be another object $B$, we write $A\coloneqq B$
for emphasis. 

We denote by $\N=\{1,2,\dots\}$ the 
set of natural numbers and by $\N_0=\{0,1, 2, \dots \}$\index{n1@$\N_0$} 
the set of natural numbers
including $0$. We write  $\Z$ for the set of integers, and $\Q$, $\R$, $\C$ for the set 
of  rational, real, and complex numbers, respectively. 
For  $k\in \N$, we let  $\Z_k= \Z/k\Z$
be the cyclic group of order $k$. 

 We also consider
$\widehat{\N}\coloneqq \N\cup\{\infty\}$. 
Given $a,b\in
\widehat{\N}$ we write $a|b$ if 
$a$ divides $b$. This notation is extended to
$\widehat{\N}$-valued  
functions\index{$\norm{Da}$@$a\vert b$}. If $A\sub \widehat{\N}$, then $\lcm(A)\in  \widehat{\N}$
denotes the least common multiple of the numbers in $A$.
 See Section~\ref{sec:orbif-assoc-thurst} for more details. 

The 
\emph{floor}\index{$\norm{D1}$@$\floor{x}, \ceil{x}$} 
of a real number $x$, denoted by $\floor{x}$, is the
largest integer $m\in \Z$ with $m\leq x$. The \emph{ceiling} of a real
number $x$, denoted by $\ceil{x}$, is the smallest integer $m\in \Z$
with $x\leq m$. 

The symbol $\iu$ stands for the imaginary unit
in the complex plane $\C$. 
The real and imaginary part of a complex number $z$ are indicated 
by $\real(z)$ and $\imag(z)$, respectively, and its complex conjugate
by $\overline{z}$. 
The open unit disk in $\C$ is denoted by $\D\coloneqq \{z\in \C:|z|<1\}$,  and the open upper half-plane by $\Halb \coloneqq \{z\in \C:\imag(z)>0\}$.


We let $\CDach\coloneqq \C\cup\{\infty\}$\index{CAA@$\CDach$} be   the Riemann sphere. It carries  the \emph{chordal metric} $\sigma$ given by  formula \eqref{eq:def_chordal} (in the appendix). 
Similarly, we let 
$\widehat{\R}\coloneqq  \R\cup \{\infty\}$.\index{RA@$\widehat{\R}$}  
Here we consider $\widehat{\R}$ as a subset of  $\CDach$, and so 
$\widehat{\R}\sub \CDach$. 

The \emph{Lebesgue measure} on $\R^2$, $\C$, $\CDach$, or $\D$
is denoted by $\leb$. If necessary, we add a subscript here to avoid ambiguities.    More precisely, $\leb=\leb_{\R^2}$ and 
$\leb=\leb_{\C}$ are  the Euclidean area measures on $\R^2$ and
$\C$, $\leb= \leb_{\CDach}$ is the spherical area measure on
$\CDach$, and $\leb=\leb_{\D}$ the hyperbolic area measure on
$\D$ considered as the hyperbolic plane. 

When we consider two objects $A$ and $B$, and there is a natural
identification between them that is clear from the context, we
write 
$A\cong B$.\index{$\cong$} 
For example, $\R^2\cong \C$ if we identify a point $(x,y)\in \R^2$ with $x+ y\iu \in \C$. 

The derivative of a holomorphic function $f$ is  denoted by $f'$
as usual. If $\Omega\subset \CDach$ is an open set and
$f\colon \Omega \to \CDach$ is a holomorphic map, then $f^\sharp$
stands for its  
\emph{spherical derivative} (see
\eqref{eq:defsphericaldf}).  For a  differentiable 
(not necessarily holomorphic)  map, we use   $Df$ to denote its derivative  considered as a linear map between suitable tangent spaces. If these tangent spaces are equipped with norms, then we let $\norm{Df}$ be the operator norm of $Df$. Sometimes we use subscripts here to indicate the norms.  

Two non-negative quantities $a$ and $b$ are said to be \defn{comparable}
if there is a constant $C\ge 1$ (possibly depending on some  ambient para\-meters)   such that
\begin{equation*}
  \frac{1}{C}a\leq b\leq C a.
\end{equation*}
We then write 
$a\asymp b$.\index{$\asymp$}\index{C@$C(\asymp)$}\index{comparable}
The constant $C$ is referred to  as $C(\asymp)$. Similarly, we write
  $a\lesssim b$ or $b\gtrsim a$,  
if there is a constant $C>0$ such that $a\leq C b$, and  refer to
the constant $C$ as $C(\lesssim)$ or
$C(\gtrsim)$.\index{$\lesssim,\gtrsim$}\index{CA@$C(\lesssim)$}
If we want to emphasize the 
parameters $\alpha$, $\beta, \dots$ on which 
$C$ depends, then we write $C(\asymp)=C(\alpha, \beta, \dots)$ etc. 

The cardinality of a set $X$ is denoted by $\#X$ and the identity map on $X$ by $\id_X$. 
If $x_n\in X$ for $n\in \N$ are points in $X$, we  denote the sequence of these points 
  by $\{x_n\}_{n\in\N}$, or just
by $\{x_n\}$ if the index set $\N$ is understood.

  If $f\: X \ra X$ is a map and $n\in \N$, then 
$$f^n\coloneqq \underbrace{f\circ \dots \circ f}_{\text{$n$ factors} } $$ is the $n$-th iterate of $f$. We set $f^0\coloneqq \id_X$ for convenience, but unless otherwise indicated  it is understood that $n\in \N$ if  we speak of an iterate $f^n$ of $f$.    

Let $f\colon X\to Y$ be a map between sets $X$ and $Y$. If
$U\sub  X$, then 
$f|U$\index{$f^\sharp a$@$f\vert U$}
stands for the \emph{restriction} of $f$ to $U$. If $A\sub Y$, then 
$f^{-1}(A)\coloneqq \{x\in X : f(x)\in A\}$ is the preimage of
$A$ in $X$. Similarly,  $f^{-1}(y)\coloneqq
\{x\in X : f(x)=y\}$ is the preimage of a point $y\in Y$. 

If $f\: X\ra X$ is a map, then preimages of a set $A\subset X$ or a point $p\in X$ under the
$n$-th iterate $f^n$ are denoted by $f^{-n}(A)\coloneqq \{x\in X
: f^n(x) \in A\}$ and $f^{-n}(p)\coloneqq \{x\in X : f^n(x)=p\}$, 
respectively. 

Let $(X,d)$ be a metric space, $a\in X$, and $r>0$. 
By
$B_d(a,r)=\{x\in X:d(a,x)< r\}$ we denote the open and by
$\overline B_d(a,r)=\{x\in X: d(a,x)\le r\}$ the 
closed ball of radius $r$ centered at $a$. 
If
$A,B\sub X$, we let $\diam_d(A)$ be the diameter, $\overline A$
be the closure of $A$ in $X$, and
$$ \dist_d(A,B)\coloneqq\inf\{d(x,y): x\in A, y\in B\}$$
be the distance of $A$ and $B$.   
If $p\in X$, we let $\dist_d(p, A)\coloneqq\dist_d(\{p\}, A)$. 
For $\eps>0$,
\index{N5@$\mathcal{N}_\epsilon$}
\begin{equation*}
  \mathcal{N}_{d,\epsilon}(A)
  \coloneqq 
  \{ x\in X: \dist_d(x,A)<\eps\} 
\end{equation*}
is the \emph{open $\eps$-neighborhood} of $A$ with respect to
$d$. If $\gamma\colon [0,1]\to X$ is a path, we denote by
$\length_d(\gamma)$ the length of $\gamma$. Given $Q\geq 0$, we
denote by $\mathcal{H}^Q_d$ the 
\emph{$Q$-dimensional Hausdorff measure}\index{H_Q@$\mathcal{H}^Q_d$}\index{Hausdorff!measure}
on $X$ with respect to $d$. 
We drop the
subscript $d$ in our notation for $B_d(a,r)$, etc., if the metric $d$ is clear
from the context. For the Euclidean metric on $\C$ we sometimes use the subscript $\C$ for emphasis. So, for example, 
$$B_\C(a,r)\coloneqq \{ z\in \C: |z-a|<r\}$$ 
denotes the Euclidean 
ball  of radius $r>0$ centered at $a\in \C$. 

The \emph{Gromov product} of two points $x,y\in X$ with respect to a 
basepoint $p\in X$ in a metric space $X$ is denoted by
$(x\cdot y)_p$ or by $(x\cdot y)$ if the basepoint $p$ is
understood (see Section~\ref{sec:Grhyp}). The \emph{boundary at
  infinity} of a Gromov hyperbolic space $X$ is represented   by $\geo
X$. If a group $G$ acts on a space $X$, then we write $G\acts X$ to indicate this action. 

Often we use the notation $I= [0,1]$. If $X$ and $Y$ are topological spaces,
then a \emph{homotopy} is a continuous map $H\colon X\times
I \to Y$. For $t\in I$, we let  $H_t(\cdot)\coloneqq  H(\cdot, t)$ be the {\em time-$t$ map} of the homotopy.

The symbol  $S^2$ indicates a $2$-sphere, which we think of as a
topological object. Similarly, $T^2$ is a  topological   $2$-torus.
 For a $2$-torus with a Riemann surface structure we write 
 $\T$ (see Section~\ref{sec:applifttorend}).

 Often  $S^2$ (or the Riemann sphere $\CDach$) 
is equipped with certain metrics that induce its  topology. The \emph{visual metric} induced
by an expanding Thurston map $f$ is usually denoted by
$\varrho$ (see Chapter~\ref{cha:visual-metrics}). The \emph{canonical orbifold metric} of 
a rational Thurston map $f$ is indicated by $\omega_f$ (see
Section~\ref{sec:expratThmaps}).

The (topological) \emph{degree} of a branched covering map $f$
between surfaces is denoted by $\deg(f)$ and the \emph{local degree} of $f$ at a point $x$ by $\deg_f(x)$
or $\deg(f,x)$ (see Section~\ref{sec:branched-coverings}). We write $\crit(f)$
for the \emph{set of critical points} of a branched covering map (see
Section~\ref{sec:branched-coverings}), and $\post(f)$ for the set of
\emph{postcritical points} of a Thurston map $f$
(see Section~\ref{sec:defin-thurst-maps}). 

The \emph{ramification function} of a Thurston map $f$ is
denoted by $\alpha_f$ (see Definition~\ref{def:weightf}), and the
\emph{orbifold} associated with  $f$  by $\OC_f$ (see
Definition~\ref{def:orbifold_f}).

For a given Thurston map $f\: S^2\ra S^2$ we usually use the symbol  
$\CC$ to  indicate  a Jordan curve $\CC\subset
S^2$ that satisfies $\post(f) \subset \CC$. 

When we consider objects that are defined in terms of
the $n$-th iterate of a given Thurston map, then we often 
use the upper index ``$n$'' to emphasize this.

For a topological cell  $c$ in a topological  space $\mathcal{X}$ we denote by $\partial c$ the \emph{boundary} of $c$,
and by $\inte(c)$ the \emph{interior} of $c$ (see
Section~\ref{s:celldecomp}). Note that $\partial c$ and
$\inte(c)$ usually do not agree with the boundary or
interior of $c$ as a subset of $\mathcal{X}$.

\emph{Cell decompositions} of a space $\mathcal{X}$ are usually denoted by $\DD$ (see
Chapter~\ref{cha:celldecomp}). Let $n\in \N_0$, $f\: S^2\ra S^2$
be a Thurston map, and $\CC\subset S^2$ be 
 a Jordan curve  with  $\post(f) \subset \CC$. We then write  $\DD^n(f,\CC)$ for  the cell
decomposition of $S^2$ consisting of the {\em cells of level $n$} or {\em $n$-cells}  defined in terms of
$f$ and $\CC$ (see
Definition~\ref{def:DDn}).  The set of corresponding \emph{$n$-tiles} is denoted by $\X^n$, the set of
\emph{$n$-edges} by $\E^n$, and the set of \emph{$n$-vertices} by
$\V^n$ (see Section~\ref{sec:tiles}).  

In this context we often  ``color'' tiles  ``black'' or ``white''. We 
then use the subscripts $\bt$ and $\wt$ to indicate the color (see the end of
Section~\ref{sec:tiles}). For example, the black and white
$0$-tiles are denoted by $\XOb$ and $\XOw$, respectively.

The \emph{$n$-flower} of an $n$-vertex $v$ is denoted by
$W^n(v)$ (see Section~\ref{sec:flowers}). 
The number $D_n=D_n(f,\CC)$ is the minimal number of $n$-tiles
required to join opposite sides (see \eqref{def:dk}). 

The number $m(x,y)=m_{f,\CC}(x,y)$ is defined in
Definition~\ref{def:mxy}. The \emph{expansion factor} of a
visual metric is usually denoted by $\Lambda$ (see
Definition~\ref{def:visual}). 

We write $\Lambda_0(f)$ for the  \emph{combinatorial expansion factor} of a 
Thurston map $f$ (see Proposition~\ref{prop:exp}). 

 The \emph{topological entropy}
of a map $f$ is denoted by $h_{top}(f)$, and the \emph{measure-theoretic entropy} of $f$ with respect to a measure $\mu$ by
$h_{\mu}(f)$. The \emph{measure of maximal entropy} of an expanding Thurston
map $f$ is indicated  by $\nu_f$. See Chapter~\ref{cha:measure} for these concepts.

For a rational Thurston map $f\: \CDach \ra \CDach$ we write $\Omega_f$ for its \emph{canonical orbifold
  measure} (see Section~\ref{sec:expratThmaps}) 
and, if $f$ is also expanding, $\lambda_f$ for 
  the unique probability measure on $\CDach$ 
that is absolutely continuous with respect to Lebesgue measure
 (see Chapter~\ref{cha:rati-thurst-maps-1}).


\mainmatter
%


\chapter{Introduction} 
\label{cha:introduction}

In this work we study the dynamics of Thurston maps under
iteration. A Thurston map is a branched covering map on a
$2$-sphere $S^2$ such that  each of its critical points has a finite orbit.
The most important examples   are given by 
postcritically-finite rational maps on the Riemann sphere $\CDach$.  Most of the time   we will also assume that a
Thurston map is expanding in a suitable sense. For postcritically-finite rational maps 
$f\:\CDach\ra \CDach$ expansion is equivalent to the requirement that $f$ does not have periodic critical points or that its Julia set is equal to $\CDach$. 

These  objects  were first considered by Thurston as topological
mo\-del maps in the context of his celebrated 
characterization  of rational maps (see Theorem~\ref{thm:Thurston}). 
The terminology was introduced by Douady and Hubbard in their
proof of this theorem.

Every expanding Thurston map $f\: S^2 \ra S^2$ gives rise to a type of
fractal geometry on the underlying sphere $S^2$. This geometry is represented by a class of \emph{visual metrics} 
$\varrho$ that are associated with the map. Many  dynamical properties of the map are encoded in the 
geometry of the corresponding {\em visual sphere}, meaning $S^2$ equipped with a visual metric $\varrho$.

For example, we will see that an expanding Thurston map is topologically conjugate to a rational map if and only if 
$(S^2, \varrho)$ is quasisymmetrically equivalent to $\CDach$ (see Section~\ref{sec:QCgeom} for the terminology). 
For us this relation between dynamics and fractal geometry is one of the main motivations for studying expanding Thurston maps. 

In order to define a 
 visual metric for a given 
Thurston map $f\colon S^2\to S^2$, we  will extract some combinatorial data from $f$. For this we consider a cell
decomposition of $S^2$ and its  pull-backs by the iterates
$f^n$. When $f$ is expanding, the diameters of the cells in these
 decompositions shrink to $0$; so we get  discrete approximations of $S^2$ that get
finer with larger level $n$. Given two distinct points  in $S^2$, one can  ask at which level
the cell decompositions will allow us to distinguish them. Our
definition of a visual metric is based on this information. 

The visual sphere $(S^2,\varrho)$ of an expanding Thurston map is fractal in the sense that its  Hausdorff dimension  is typically larger than $2$. With a suitable choice of $\varrho$, the local behavior
of $f$ becomes very simple though.  Namely, there is a number $\Lambda>1$
(the \emph{expansion factor} of $\varrho$) such that $f$
expands $\varrho$ locally by the  factor $\Lambda$ in a
sense that will be made precise. So 
the local
behavior of $f$ on $(S^2,\varrho)$ is simplified at the expense
of a more complicated geometry of  $(S^2,\varrho)$.
This point of view  is in contrast to   the usual setting for complex dynamics, where one studies  the action of a rational map on a {\em smooth} underlying space, namely the  Riemann sphere $\CDach$, considered as a Riemann
surface. 

It is possible to construct  a graph $\mathcal{G}$ that combines  the combinatorial data of the cell decompositions on all levels generated by an expanding Thurston map and its iterates. This graph $\mathcal{G}$ is Gromov hyperbolic and its
boundary at infinity can  naturally be identified with
the underlying sphere $S^2$. Under
this identification a metric is a visual metric for the given map $f$ according to our
definition if and only if it is a visual metric in the sense of 
Gromov hyperbolic
spaces. This fact relates the study  of 
expanding Thurston maps and of Gromov hyperbolic spaces.

There is an intriguing connection of these ideas to 
 \emph{Cannon's
  conjecture} in geometric group theory. Roughly speaking, this conjecture predicts that a group $G$
that shares the topological properties of the fundamental group of a
closed hyperbolic $3$-manifold ``is'' such a fundamental group (see
Section~\ref{sec:Cannconj} for precise statements). In this context one assumes that the group $G$ is Gromov hyperbolic  and that its  boundary at infinity $\partial_\infty G$ is  a $2$-sphere. Here  $\partial_\infty G$ is naturally equipped with a visual metric that 
provides $\partial_\infty G$ with a fractal  geometry.
 Then Cannon's conjecture is 
equivalent to showing that the  fractal sphere $\partial_\infty G$  is quasisymmetrically equivalent to $\CDach$. 

So for both types of dynamical systems, namely expanding Thurston maps and Gromov hyperbolic groups $G$ with $2$-sphere boundary 
$\partial_\infty G$, we are led to the investigation of a fractal geometry
on the underlying $2$-sphere. This analogy can be  viewed as an
example of {\em Sullivan's dictionary} which exhibits similarities in complex 
dynamics and the theory of Kleinian groups.
Common  to  both areas is the desire to characterize conformal  dynamical systems in a wider class of dynamical systems characterized by suitable metric-topological conditions. One should not  push the analogies too far though:  while Cannon's conjecture is generally believed to be true and, accordingly, one expects that the fractal $2$-spheres arising from Gromov hyperbolic groups are always quasisymmetrically equivalent to $\CDach$, this is not always the case  for  Thurston maps, because not every 
Thurston map is equivalent to a rational map. 

After these remarks about some of  the motivations for our investigation,  we now   state some basic
definitions more precisely (more details can be found in Chapter~\ref{cha:thmaps}).
Let $f\:S^2\ra S^2$ be an (orientation-preserving) branched
covering map. As usual, we call a point
$c\in S^2$ a {\em critical point} of $f$ if near $c$ the map $f$
is not a local homeomorphism. A {\em postcritical point} is any
point obtained as an image of a critical point under forward
iteration of $f$. So if we denote by $\crit(f)$ the set of
critical points of $f$ and by $f^n$ the $n$-th iterate of $f$,
then the set of postcritical points of $f$ is given by
\begin{equation*}
  \post(f)\coloneqq \bigcup_{n\geq 1} \{f^n(c):c\in \crit(f)\}. 
\end{equation*}

It is a fundamental fact in complex dynamics that much
information on the dynamics can be deduced from the
structure of the orbits of critical points. A very strong
assumption in this respect is that each such orbit is finite,
i.e., that  $\post(f)$ is a finite set. In this case the
map $f$ is called {\em postcritically-finite}. A {\em Thurston map} is
a (non-homeomorphic) branched covering map $f\colon S^2\to S^2$ that is
postcritically-finite. 

 Thurston maps are abundant and include specific
{\em rational Thurston maps} (i.e., rational maps on $\CDach$
that are postcritically-finite)  such as $f(z)=1-2/z^2$ or
$f(z)=1+(\iu-1)/z^4$. More examples can be found in
Section~\ref{sec:examples-two-tile}, and a list of
examples considered in this book is given in
Section~\ref{sec:examples}.  We will later provide  a general method
for producing Thurston maps (see
Pro\-po\-si\-tion~\ref{prop:rulemapex}); it follows from one of our main
results (Theorem~\ref{thm:main}) that at least some iterate of
every expanding Thurston map can be obtained from this
construction.

We now turn to the discussion of more specific topics in this
introductory chapter. Our main purpose is to give some guidance
for the intuition of the reader. We will present some examples
and discuss the main concepts and results of this work.  Full
details can be found in subsequent chapters.

\section{A Latt\`es map as a first  example} 
\label{sec:Lattes}
\index{Latt\`{e}s map}

Latt\`es
  maps form a  large class of well-understood Thurston maps.
    They are rational maps obtained as quotients of
holomorphic torus endomorphisms.  They were the first known examples of
rational maps whose Julia set is the whole sphere. We will discuss these maps in
more detail in Chapter~\ref{cha:lattes-lattes-type}; results
concerning them will be outlined in
Section~\ref{sec:char-latt-maps}. Note that the terminology is
not uniform and some authors use the term Latt\`es map with a
slightly different meaning.

We will encounter Latt\`{e}s maps quite often in
this book. On the one hand, they
are easy to visualize and construct, and thus often  serve as convenient examples to
illustrate various phenomena. On the other hand, these maps are
quite special and arise in many situations as exceptional cases. 
In order  to introduce   some of the main themes of this work, we
will now consider a specific  Latt\`es map.

 \ifthenelse{\boolean{nofigures}}{}{
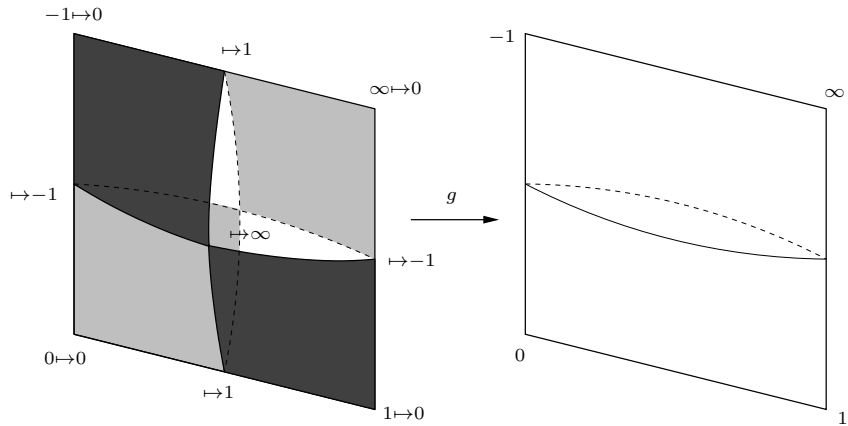
\begin{figure} 
  \centering
  \includegraphics[scale=0.5]{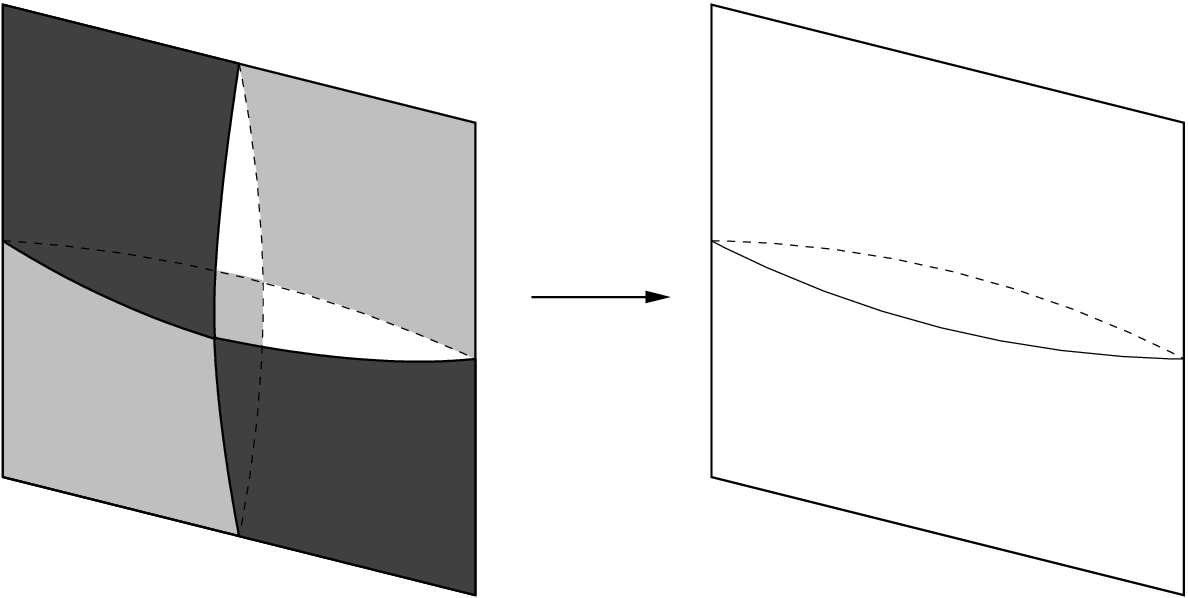}
  \begin{picture}(10,10)
    \put(-122,19){$\scriptstyle 0$}
    \put(0,-5){$\scriptstyle 1$}
    \put(-132,140){$\scriptstyle -1$}
    \put(-5,119){$\scriptstyle \infty$}
    \put(-172,-3){$\scriptstyle 1\mapsto 0$}
    \put(-240,5){$\scriptstyle \mapsto 1$}
    \put(-300,18){$\scriptstyle 0\mapsto 0$}
    \put(-313,80){$\scriptstyle \mapsto -1$}
    \put(-300,148){$\scriptstyle -1\mapsto 0$}
    \put(-233,135){$\scriptstyle \mapsto 1$}
    \put(-177,120){$\scriptstyle \infty \mapsto 0$}
    \put(-170,55){$\scriptstyle \mapsto -1$}
    \put(-148,80){$\scriptstyle g$}
    \put(-230, 65){$\scriptstyle \mapsto \infty$}
  \end{picture}
  \caption{The Latt\`{e}s map $g$.}
  \label{fig:mapg}
 \end{figure}
}

The map is essentially given 
 by Figure~\ref{fig:mapg}. We will explain this picture in detail momentarily, but we will first define the map by a more standard approach. This may be helpful for  readers that 
are   already  familiar with
  Latt\`es maps.

  The square $[0,\frac{1}{2}]^2\sub \R^2\cong \C$ can be
 mapped conformally to the upper half-plane in $\CDach$ such that
the vertices
$0,\frac{1}{2}, \frac{1}{2} + \frac{\iu}{2}, \frac{\iu}{2}$ of
the square are mapped to the points $0,1,\infty,-1$,
respectively. Note that here and in the following a ``conformal
map'' is always bijective.  
By Schwarz reflection we can extend this
to a holomorphic map $\Theta\colon \C\to \CDach$. Up to
postcomposition with a M\"{o}bius transformation, this map is a
classical \defn{Weierstra\ss\ $\wp$-function}; it is doubly-periodic with respect to the lattice 
$\Gamma \coloneqq \Z\oplus \Z\iu$ and induces a double branched
covering map of the torus $\T\coloneqq \C/\Gamma$ to the sphere
$\CDach$.

Consider the map
\begin{equation*}
  A\colon \C\to \C,
  \quad u \mapsto A(u)\coloneqq  2 u.
\end{equation*}
From the properties of the $\wp$-function or directly from the definition of $\Theta$ by the reflection process, one can see that  $\Theta(v)=\Theta(u)$ for $u,v\in \C$ if and only if 
$v=\pm u+\ga $ with $\ga\in \Gamma$. In this case,  $\Theta(2v)=\Theta(2u)$. 
This implies 
 that there is a well-defined and unique holomorphic map $g\colon \CDach\to
\CDach$ such that the diagram
\begin{equation}
  \label{eq:Lattes}
  \xymatrix{
    \C \ar[r]^A \ar[d]_{\Theta} &
    \C \ar[d]^{\Theta}
    \\
    \CDach \ar[r]^g & \CDach
  }
\end{equation}
commutes. The map $g$ obtained in this way is a \emph{Latt\`es
  map}. It is a rational map. 
One can show that it is   given by 
$$g(z)=4\frac{z(1-z^2)}{(1+z^2)^2} \quad \text{ for } z\in \CDach, $$ 
and that  the Julia set of $g$ is the whole
sphere.

More relevant for us than this explicit formula for $g$ is that
one can describe $g$ geometrically as indicated 
in Figure~\ref{fig:mapg}. To explain this, note that  there is an
essentially unique path metric on $\CDach$ obtained as a ``push-forward''
of the Euclidean metric on $\C$
by the map $\Theta$. This metric is in fact
the 
\emph{canonical orbifold metric}\index{canonical orbifold!metric}\index{metric!canonical orbifold}\index{o@$\omega$}\index{orbifold!canonical metric}\index{push-forward!of metric!by orbifold covering map}
of $g$ (see
Section~\ref{sec:expratThmaps} and Section~\ref{sec:orbif-assoc-thurst}). 

Geometrically, the sphere equipped with this metric looks like a
pillow.  In general, a \emph{pillow}\index{pillow} (see Section~\ref{sec:expratThmaps})   is a metric
space $P$ obtained from gluing two identical copies $X_{\tt w}$ and $X_{\tt b}$  of a (simple and  compact)  Euclidean polygon $X\sub \C$
  together along their boundaries. The pillow is equipped with the induced path metric. Under the given  identification,
   $\partial X_{\tt w} \cong \partial X_{\tt b}$ is a Jordan
   curve in the pillow $P$ called its 
{\em equator}.\index{equator}\index{pillow!equator of}  
   
 In our case, the upper and lower half-planes in
$\CDach$ equipped with the canonical orbifold metric are isometric to
copies of the square $S=[0,1/2]^2$.  If we glue two copies
of $S$ together  along their boundaries, then we obtain the pillow
$P$. We color one of these squares, say the one corresponding to the
upper half-plane, white, and the other square black.

The square $S=[0,1/2]^2\subset \R^2\cong \C$ (and each of its
translates by $\frac{1}{2}(m+n\iu)$ where $m,n\in \Z$) can be subdivided
into four squares of side length $1/4$.  If $S'$ is such a square,
then $A(S')$ is a square of side length $1/2$ that is mapped by
$\Theta$ to either the upper or the lower half-plane, meaning to
either the black or the white face of the pillow $P$. It follows
from \eqref{eq:Lattes} that $g$ has a very similar mapping
behavior on $P$.

More precisely, we divide each of the two sides of $P$ (each of the two
isometric copies of $S$ contained in $P$) into four smaller
squares of half the side length, and color the eight small
squares in a checkerboard fashion black and white.
 If we map one
such small white square to the large white square by a Euclidean
similarity (that scales by the factor $2$), then this map extends by reflection to the whole
pillow. There are obviously many different ways to color and map
the small squares.  If we do this in an appropriate way as
indicated in Figure \ref{fig:mapg}, then we obtain the map $g$.




The vertices  where four  small squares intersect  are the critical points of
$g$. They are mapped by $g$ to the set $\{1,\infty,-1\}$, which in turn is mapped to 
$\{0\}$.  The point $0$ is a fixed point of $g$.  So  $g$ is a  postcritically-finite branched 
covering map on the $2$-sphere $P$ 
 with $\post(g)=\{0,1,\infty,-1\}$, and hence a Thurston map.   The postcritical points of $g$ are the vertices of the pillow, which are the conical singularities of our canonical orbifold metric. The extended real line $\CC\coloneqq \widehat{\R}=\R\cup\{\infty\}$ (corresponding to the 
 equator of the pillow) is a Jordan curve that is invariant under $g$ in the sense that $g(\CC)\sub \CC$ and contains   
 the set $\{0,1,\infty,-1\}$ of  postcritical points of $g$.
The set   $g^{-1}(\CC)$ is an embedded graph in the pillow consisting  of all sides  of the small squares on the left hand
side of Figure \ref{fig:mapg} as edges and the points in $g^{-1}(\post(g))$, i.e., the corners of these squares,  as vertices. This graph $g^{-1}(\CC)$ 
determines  the tiling in  this  picture. 

The set $g^{-2}(\CC)$ is obtained by pulling $g^{-1}(\CC)$ back
by the map $g$. Since $g$ restricted to any small square $S'$ is
a homeomorphism onto one of the two large squares $S$ forming the
pillow, in this process $S'$ is subdivided in the same way as $S$
was subdivided by the small squares of side length $1/4$ (i.e.,
$S'$ is subdivided into $4$ squares). It
follows that $g^{-2}(\CC)$ subdivides the pillow into
$4\times 8=32$ squares of side length $1/8$. Proceeding in this
way inductively, we see that the preimage $g^{-n}(\CC)$ of $\CC$ under  the iterate $g^n$ subdivides the pillow
into $2\cdot 4^{n}$ squares of side length $2^{-n-1}$ for $n\in
\N$. 

The complementary components of $g^{-n}(\CC)$ are the interiors
of these squares. In particular, the diameters of these
components tend to $0$ uniformly as $n\to \infty$. This fact will be the basis of our definition 
of   an \emph{expanding}  Thurston map. Accordingly, $g$ is such a map.

 For each $n\in \N$ the set  $g^{-n}(\CC)$ forms an embedded
 graph in the pillow $P$ with the points in $g^{-n}(\post(g))$
 as vertices. This is also meaningful for $n=0$, if we
 interpret  $g^0$ as the identity map on the pillow $P$. Then this graph is just the Jordan curve $\CC$ with the points in $\post(g)$ 
as vertices.

The graph $g^{-n}(\CC)$ is the {$1$-dimensional skeleton} or
{$1$-skeleton} of a \emph{cell decomposition}
$\DD^n =\DD^n(g,\CC)$ of the pillow $P$ generated by $g$ and
$\CC$ (see Chapter~\ref{cha:celldecomp} for the terminology that
we use here and below). The $2$-dimensional cells or {\em tiles}
of the cell decomposition $\DD^n$ are squares of side length
$2^{-n-1}$ and are given by the closures of the complementary
components of $g^{-n}(\CC)$ in $P$. The map $g$ sends  each cell in
$\DD^{n+1}$ homeomorphically to a cell in $\DD^n$ (for all
$n\in \N_0$);  so $g$ is \emph{cellular} for each pair $(\DD^{n+1},\DD^n)$ of cell decompositions.

Since  $\CC$ is $g$-invariant in the sense that
$g(\CC)\subset \CC$, we have $g^{-n}(\CC)\sub
g^{-(n+1)}(\CC)$ for each $n\in \N_0$. 
This inclusion for $1$-skeleta implies that the  cell
decomposition $\DD^{n+1}$ is a {\em refinement} of $\DD^n$. On a more intuitive level, this means that  the tiles in $\DD^n$ are
subdivided by the tiles in $\DD^{n+1}$.
 
The tiles in $\DD^0$ are the two initial squares of
side length $1/2$ forming the pillow, and $\DD^1$ is formed by
squares of side length $1/4$ subdividing these squares. Since we
repeat the same subdivision procedure in the passage from $\DD^n$
to $\DD^{n+1}$, this whole sequence of cell decompositions $\DD^n$
is essentially generated by the initial pair $(\DD^1,\DD^0)$.
This pair $(\DD^1,\DD^0)$ is a {\em cellular Markov
  partition} for $g$ (see Definition~\ref{def:cellular}). The map $g$ sends each cell in $\DD^1$ to a cell in
$\DD^0$. The cellular Markov partition $(\DD^1,\DD^0)$ together
with this information completely determines the map $g$ (up to
conjugation). 
In this sense, the dynamics 
of $g$ is described by
finite combinatorial data.

In fact, one can  turn this process around and  construct $g$  from
this combinatorial data, meaning essentially from the information encoded in 
Figure~\ref{fig:mapg}. A related discussion  can be found in
Section~\ref{sec:int-inv-cell-de}. 

\section{Cell decompositions}
\label{sec:cell-decomposition}

The previous example motivates several concepts for a general
Thur\-ston map $f\: S^2\ra S^2$, in particular the combinatorial description of
$f$ that we will employ. 

We choose a Jordan curve $\CC\sub S^2$ with $\post(f)\sub \CC$,
and consider the preimages $f^{-n}(\CC)$. Then for each
$n\in \N_0$ one obtains an associated cell decomposition
$\DD^n=\DD^n(f,\CC)$ of $S^2$.\index{d00@$\DD^n(f,\CC)$}\index{cell!decomposition} 
Its vertices are the points in
$f^{-n}(\post(f))$, and its $1$-skeleton the set $f^{-n}(\CC)$.
The condition $\post(f)\sub \CC$ ensures that the closure of each
complementary component of $f^{-n}(\CC)$ is a closed Jordan
region. These sets are the $2$-dimensional cells in $\DD^n$. We
call each such set a {\em tile of level} $n$ or an $n$-{\em tile}
(of the cell decomposition). Similarly, we call any point
$v\in f^{-n}(\post(f))$ a \emph{vertex} of level $n$ or an
\emph{$n$-vertex}; then $\{v\}$ is a $0$-dimensional cell in
$\DD^n$. Finally, the closure $e$ of a component of
$f^{-n}(\CC)\setminus f^{-n}(\post(f))$ is called an \emph{edge}
of level $n$ or an \emph{$n$-edge}; then $e$ is a $1$-dimensional
cell of $\DD^n$.  The cells in $\DD^n(f,\CC)$ of any dimension are called the $n$-{\em cells}
for given $f$ and $\CC$. Note that here  $n$ always refers to the level of the cell and not to its dimension. 

The cell decomposition $\DD^0$ contains two tiles (the two closed
Jordan regions in $S^2$ bounded by $\CC$), $k=\#\post(f)$
vertices  (the points  $p\in \post(f)$), and  $k$
edges (the closed arcs into which the points in $\post(f)$ divide $\CC$). We will study 
cell decompositions and their relation to Thurston maps in more detail 
in  Chapter~\ref{cha:celldecomp}. Various examples for the cell decompositions $\DD^n$ generated in this 
way can be found in Figures~\ref{fig:def_visual_metric}, ~\ref{fig:sub_div_tiles},~\ref{fig:tiles_bary}, and~\ref{fig:invC_constr}.

We say that a Thurston map $f\: S^2\ra S^2$ is {\em expanding} if
there exists a Jordan curve $\CC\subset\Sp$ with
$\post(f) \subset \CC$ such that the complementary components of
$f^{-n}(\CC)$ become uniformly small in diameter as
$n\to \infty$. Here $S^2$ is equipped with any metric inducing
the topology on $S^2$. It is easy to see that this condition is
independent of the choice of this base metric. Later we will 
show that it is also independent of the choice of $\CC$ and will give
other characterizations of expansion (see
Chapter~\ref{cha:expansion}, in particular
Proposition~\ref{prop:expequivexp}). A rational Thurston map $f\: \CDach \ra \Cdach$ is expanding precisely  if it does not have periodic critical points or if its Julia set is equal to $\CDach$ (see Proposition~\ref{prop:rationalexpch}).
Note that in general a (non-rational) expanding Thurston map may have 
periodic critical points (see Example~\ref{ex:barycentric}).

Put differently,  a Thurston map $f$ is expanding if and only if the tiles in
$\DD^n=\DD^n(f,\CC)$ shrink to $0$ in diameter uniformly as
$n\to \infty$. This allows us to describe points in $S^2$ by
suitable sequences of tiles. So we can  think of $\DD^n$ as a discrete
approximation of $S^2$ that becomes finer with larger $n$.

 We have seen that for the example $g$ discussed in the
previous section the $n$-tiles become uniformly small in diameter as $n\to \infty$; so we conclude that $g$ is an expanding
Thurston map.

Often the precise choice of the Jordan curve $\CC$ with
$\post(f)\sub \CC$ will play no essential role, meaning that we
may chose any such curve for our considerations.  If the curve
$\CC$ is not $f$-invariant, 
then in general the cell decompositions
$\DD^n$ will not be compatible for different levels $n$.  Without
precise knowledge of the behavior of the map we will then not
have any information on how $(n+1)$-cells and $n$-cells
intersect.  The situation changes when the Jordan curve
$\CC\sub S^2$ with $\post(f)\sub \CC$ is $f$-invariant. We will
discuss existence of such invariant Jordan curves and the
resulting combinatorial description of Thurston maps in
Section~\ref{sec:int-inv-cell-de}.


\section{Fractal spheres} 
\label{sec:int-frac-sph}
 
We  want to motivate other important concepts of our
investigation, in particular the concept of visual metrics. To do
this, we will discuss another Thurston map and an associated
fractal $2$-sphere. As our main purpose here is to provide the
reader with some intuition on the definition of a visual metric
$\varrho$ and on  the fractal nature of the
sphere $(S^2,\varrho)$, we will omit the justification of some
details.

The map arises from a geometric construction that is similar to
the one used to describe the Latt\`es map in Section
\ref{sec:Lattes}. Again we start with a pillow obtained by gluing
together two squares along their boundaries; see the top right of
Figure~\ref{fig:1flap_both}.  This pillow is a polyhedral surface
$\mathcal{S}^0$ homeomorphic to the $2$-sphere.  It carries a
natural cell decomposition $\DD^0$ with the two squares as tiles, the
four sides of the common boundary of the squares as edges, and the four common
corners of the squares as vertices.  To distinguish them from
other topological cells that we will introduce momentarily, we
consider them as cells of level $0$ and accordingly call them
$0$-tiles, $0$-edges, and $0$-vertices.  As in the example of
Section~\ref{sec:Lattes}, we assign colors to the tiles; say the
top square  of $\mathcal{S}^0$ as shown in
Figure~\ref{fig:1flap_both} is white, and the bottom square is
black.

  \ifthenelse{\boolean{nofigures}}{}{



}

To obtain cells on the next level $1$, each of the two squares,
or more precisely $0$-tiles, is divided into four squares of half
the side length. We call these eight smaller squares tiles of
level $1$, or simply $1$-tiles. The edges of these squares are
$1$-edges.  We slit the sphere along one such $1$-edge in the
white $0$-tile and glue in two small squares at the slit, as
indicated on the upper left in  Figure~\ref{fig:1flap_both}. This
gives two additional $1$-tiles and we obtain a polyhedral surface
$\mathcal{S}^1$ homeomorphic to the $2$-sphere.  The surface
$\mathcal{S}^1$ carries a cell decomposition  given by the
topological cells of level $1$ as described.  We color the
$1$-tiles black and white in a checkerboard fashion so that 
 $1$-tiles sharing an edge have different color, as indicated
in Figure~\ref{fig:1flap_both}.

To define a Thurston map based on our construction, we choose an
identification of the polyhedral surface $\mathcal{S}^1$ with
$\mathcal{S}^0$. To do this, we represent the six $1$-tiles that
replaced the white $0$-tile topologically as subsets of this
tile, and similarly the other four $1$-tiles as subsets of the
black $0$-tile. So the $0$-tiles are ``subdivided'' by
$1$-tiles. This is indicated on the lower left in 
Figure~\ref{fig:1flap_both}. Under this identification the cell
decomposition of $\mathcal{S}^1$ gives a cell decomposition $\DD^1$ 
on $\mathcal{S}^0$ that is a refinement of the cell
decomposition $\DD^0$. 

Now we can define a Thurston map as follows.  We map each white
$1$-tile on the polyhedral surface $\mathcal{S}^1$ to the white
$0$-tile in $\mathcal{S}^0$, and each black $1$-tile in
$\mathcal{S}^1$ to the black $0$-tile in $\mathcal{S}^0$ by a
similarity map (preserving orientation). This is a well-defined
and uniquely determined map on $\mathcal{S}^1$ if we do this so
that the $1$-vertices marked by a black or  a white dot on the
upper left in  Figure~\ref{fig:1flap_both} are sent to
$0$-vertices in the upper right of the picture with the same
markings.

\begin{figure}
  \centering 
  \begin{overpic} 
    [width=12cm, 
    tics =10]{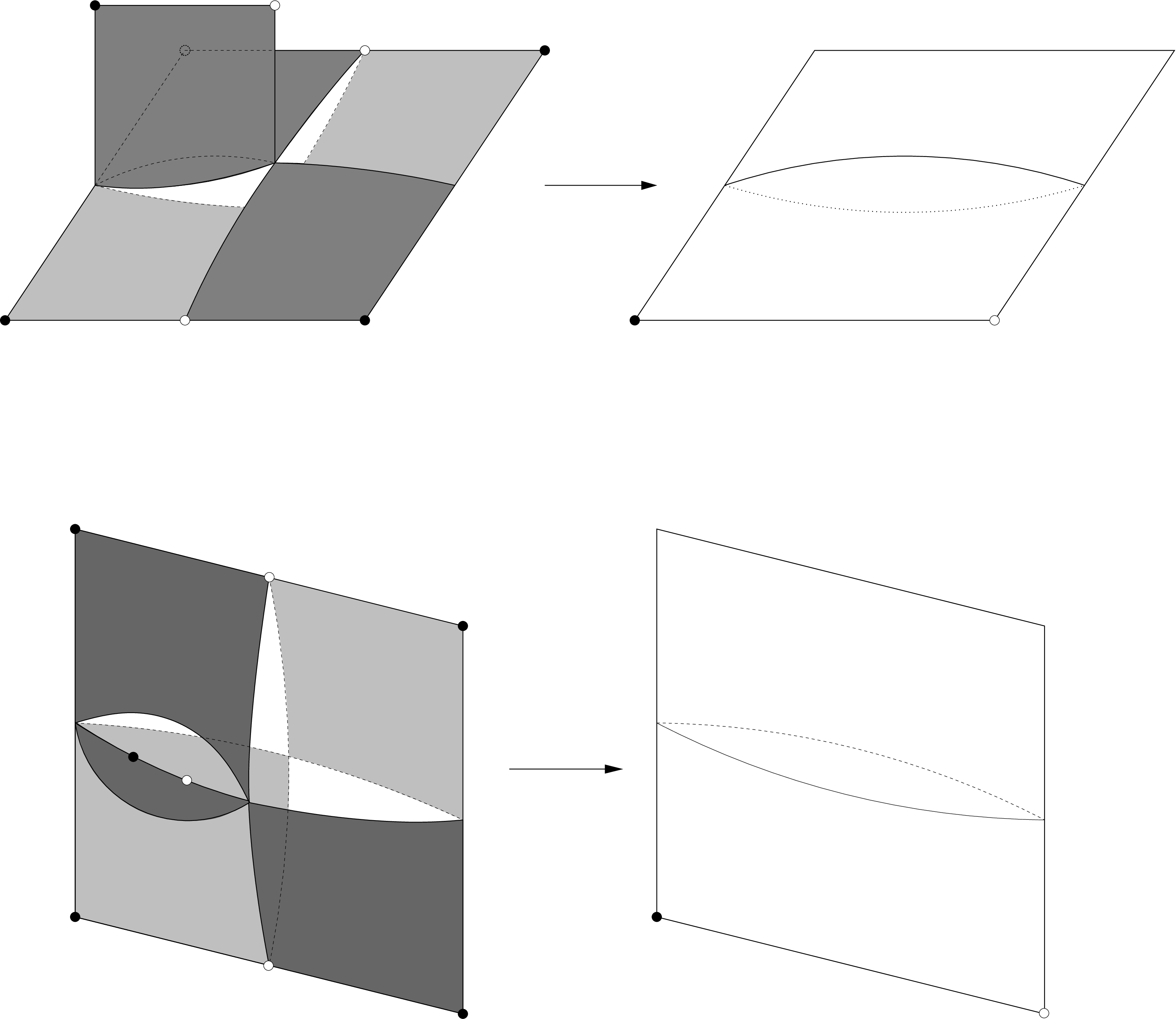}
    \put(91,62){$\mathcal{S}^0$}
    \put(37,62){$\mathcal{S}^1$}
    \put(47,23){$h$}
  \end{overpic}
  \caption{The map $h$.}
  \label{fig:1flap_both}
\end{figure}

If we identify $\mathcal{S}^1$ with $\mathcal{S}^0$ as discussed,
we get a map $h\: S^2\ra S^2$ on the $2$-sphere
$S^2\coloneqq \mathcal{S}^0$. Since $h$ restricted to each
$1$-tile is a homeomorphism onto a $0$-tile, $h$ is a branched
covering map.  The critical points of $h$ are the $1$-vertices
where at least four $1$-tiles intersect. These critical points
are all mapped to vertices of the pillow, i.e., to
$0$-vertices. All $0$-vertices in turn are mapped to the
$0$-vertex marked black. Thus $h$ is a postcritically-finite branched covering map on $S^2$, i.e.,
a Thurston map. Note that the equator $\CC$ of the pillow is an
$h$-invariant Jordan curve (i.e., $h(\CC)\subset \CC$) and that
the cell decomposition $\DD^1$ on $\mathcal{S}^1\cong \mathcal{S}^0$ is
determined by $h^{-1}(\CC)$. Namely, $h^{-1}(\CC)$ is a
topological graph that gives the $1$-skeleton of this cell
decomposition, and each $1$-tile is the closure of a
complementary component of $h^{-1}(\CC)$.

The relevant information on the map $h$ is contained in the combinatorics of the cell decompositions $\DD^0$ and $\DD^1$ and 
a map $L\: \DD^1\ra \DD^0$ that records how $h$ associates the cells 
in $\DD^1$ with cells in $\DD^0$. This triple $(\DD^1, \DD^0, L)$ forms 
a {\em two-tile subdivision rule} that is {\em realized} by $h$.
We will give precise definitions of these concepts in Chapter~\ref{cha:subdivisions}.  The map $h$ depends on choices and is not uniquely determined, but another map realizing the same 
subdivision rule $(\DD^1, \DD^0, L)$ is {\em Thurston equivalent} to 
$h$ (see Definition~\ref{def:Thequiv} for the terminology).

There is no rational map that realizes this combinatorial
picture as our map $h$. More precisely, $h$ is not Thurston
  equivalent to a rational map, because $h$ has a \emph{Thurston
  obstruction}. This, together with the terminology, will be
explained in Section~\ref{sec:thurst-class-rati}.

We 
will 
now  describe a fractal sphere $\mathcal{S}$ that is
associated with our construction  and gives an alternative way to
view our map $h$.  The sphere $\mathcal{S}$  is obtained  
similarly
 as  the well-known snowflake curve.
We will also  define a metric $\varrho$ on $\mathcal{S}$.

To construct the space $\mathcal{S}$, we do not identify the
surfaces $\mathcal{S}^0$ and $\mathcal{S}^1$. Instead, we consider
the passage from $\mathcal{S}^0$ to $\mathcal{S}^1$ as a
replacement procedure. The white $0$-tile is replaced with the top
part of $\mathcal{S}^1$, consisting of six $1$-tiles, i.e.,
squares of side length $1/2$; we call this top part of
$\mathcal{S}^1$ the \emph{white generator}.  Similarly, the four
$1$-tiles subdividing the black $0$-tile form the \emph{black
  generator}. The polyhedral surface $\mathcal{S}^1$ consisting
of ten squares is the first approximation of the fractal space
$\mathcal{S}$ that we are about to construct by iterating this
procedure.

Namely, the $1$-tiles of the black and white generators are also
colored as indicated in Figure~\ref{fig:1flap_both}. So if we
replace each black or white $1$-tile with a suitably scaled copy of
the black or white generators, then we obtain a polyhedral
surface $\mathcal{S}^2$ glued together from squares of
side length $1/4$ as $2$-tiles.  Here we have to be careful about
how precisely a tile is replaced with  an appropriate generator,
because the generators with their given colorings of tiles are
not symmetric with respect to rotations.  To specify the
replacement rule uniquely, we use the additional markings of some
points. Each generator carries two points corresponding to the
points on $\mathcal{S}^0$ marked black or white. In the
replacement process we require that these points match the
corresponding points with the same markings on $1$-tiles.

\begin{figure}
  \centering
  \begin{overpic}
    [width=12cm, 
    tics=10]{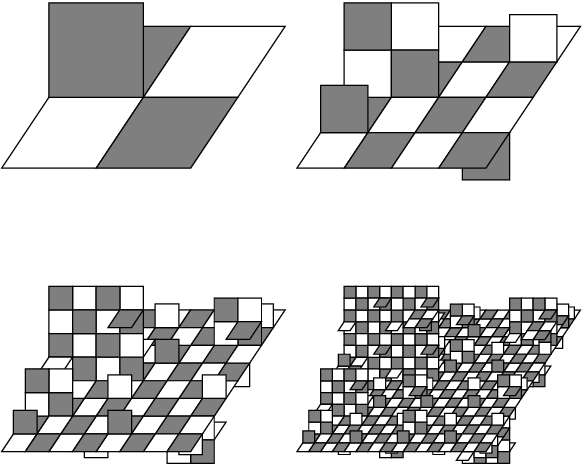}
    \put(39,56){$\mathcal{S}^1$}
    \put(90,56){$\mathcal{S}^2$}
    \put(40,6){$\mathcal{S}^3$}
    \put(90,6){$\mathcal{S}^4$}
  \end{overpic}
  \caption{Polyhedral surfaces obtained from the replacement rule.}
  \label{fig:intpolysur}
\end{figure}
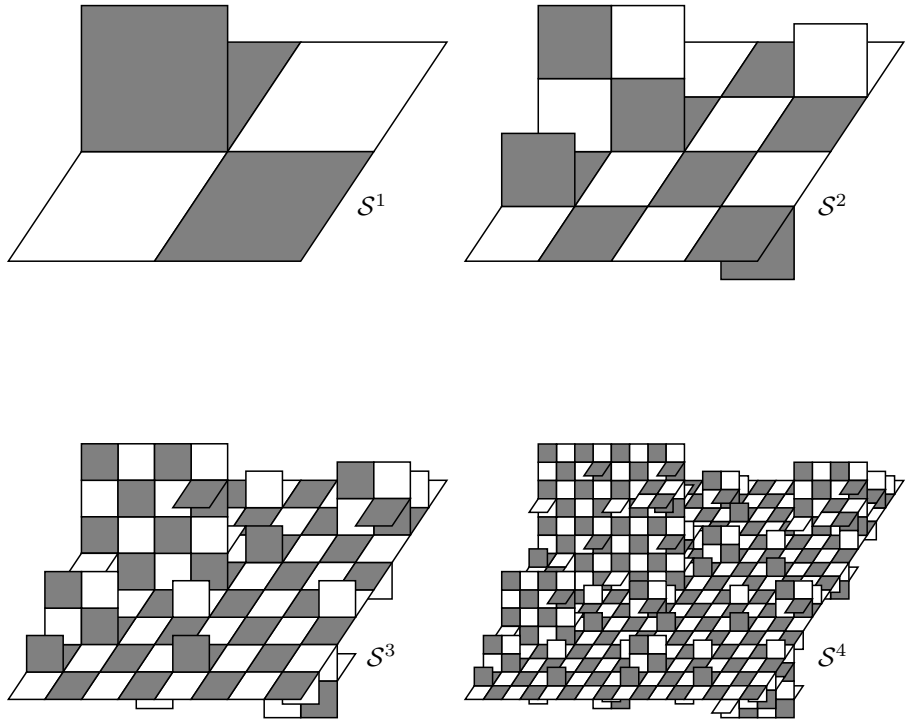

If we iterate the replacement procedure in this way, we obtain
polyhedral surfaces $\mathcal{S}^n$ for all levels $n\in \N_0$
glued together from squares of side length $1/2^n$.  Each
surface $\mathcal{S}^n$ carries a natural cell decomposition
$\DD^n$ given by these squares as tiles.  Some iterates of this
construction are shown in Figure~\ref{fig:intpolysur}.  The
pictures essentially indicate the gluing pattern of the squares
which give the surfaces. One should view them as abstract
polyhedral surfaces, and not confuse them with the underlying
subsets of $\R^3$ in these pictures.  Each surface
$\mathcal{S}^n$ is a topological $2$-sphere and carries a
piecewise Euclidean path metric $\varrho_n$ with conical
singularities.

One can now extract a self-similar ``fractal'' space
$\mathcal{S}$ as a limit $\mathcal{S}^n\to \mathcal{S}$ for
$n\to \infty$ in several ways. One possibility is to pass to a
Gromov-Hausdorff limit of the sequence
$(\mathcal{S}^n, \varrho_n)$ of metric spaces. We will discuss a
different method that is closer in spirit to our general
definition of a visual metric (see
Chapter~\ref{cha:visual-metrics} and Chapter~\ref{cha:Gromov};
similar considerations appear in Chapter~\ref{cha:combexp}).

Namely, given an $n$-tile $\mathcal{X}^n\subset \mathcal{S}^n$
(which is a square of side length $2^{-n}$), and an $(n+1)$-tile
$\mathcal{X}^{n+1}\subset \mathcal{S}^{n+1}$, we write
$\mathcal{X}^n \supsim \mathcal{X}^{n+1}$ if $\mathcal{X}^{n+1}$
is contained in the scaled copy of a generator that replaced
$\mathcal{X}^n$ in the construction of $\mathcal{S}^{n+1}$ from
$\mathcal{S}^n$. We now consider descending sequences
$\mathcal{X}^0 \supsim \mathcal{X}^1 \supsim \mathcal{X}^2
\supsim \dots$\,.
On an intuitive level the squares in such a sequence should
shrink to a point in our desired limit space $\mathcal{S}$
represented by the sequence.  Here we consider two sequences
$\{\mathcal{X}^n\}$ and $\{\mathcal{Y}^n\}$ as equivalent and
representing the same point if
$\mathcal{X}^n \cap \mathcal{Y}^n\ne \emptyset$ for all
$n\in \N_0$. It is not hard to see that this indeed defines an
equivalence relation for descending sequences. By definition our
limit space $\mathcal{S}$ is now the set of all equivalence
classes. 

For  $x,y\in 
\mathcal{S}$ we set 
\begin{equation}
  \label{eq:def_visual_1}
  \varrho(x,y)\coloneqq \limsup_{n\to \infty} \dist_{\varrho_n} (\mathcal{X}^n,  \mathcal{Y}^n),
\end{equation}
where $\{\mathcal{X}^n\}$ and $\{\mathcal{Y}^n\}$ are sequences representing $x$ and $y$, respectively. Then $\varrho$ is well-defined and one can show that this is a metric on $\mathcal{S}$. 

For $x,y\in \mathcal{S}$, $x\ne y$, we define
\begin{equation}
  \label{eq:intromxy}
  m(x,y) \coloneqq \inf\{n\in \N: \mathcal{X}^n\cap \mathcal{Y}^n=\emptyset\}, 
 \end{equation}
where  the infimum is taken over  all sequences $\{\mathcal{X}^n\}$ and $\{\mathcal{Y}^n\}$ representing $x$ and $y$, respectively.  Then 
\begin{equation}
  \label{eq:def_visual_introsec}
  \varrho(x,y) \asymp 2^{-m(x,y)}
\end{equation}
for $x,y\in \mathcal{S}$, $x\ne y$. This notation (which will be
used frequently) means that there is a constant $C\ge 1$ such that
\begin{equation*}
  \frac{1}{C} \varrho(x,y) 
  \leq 
  2^{-m(x,y)} 
  \leq C \varrho(x,y).
\end{equation*}
We refer to the constant $C$ as $C(\asymp)$ in such inequalities. In the present case,  $C(\asymp)$ does
not depend on $x$, $y$, or $n$. So roughly speaking, the distance of
two distinct points in $\mathcal{S}$ is given in terms of the minimal
level on which two descending sequences representing the points
can distinguish them. 

It is intuitively clear that $(\mathcal{S}, \varrho)$ is a
topological $2$-sphere.  To outline a rigorous proof for this
fact, we return to the Thurston map $h$ defined above. Recall
that $\CC\subset S^2$ is the
$h$-invariant Jordan curve given by the common boundary of the
$0$-tiles. The curve $\CC$ contains the vertices of the pillow,
which are the postcritical points of $h$. 
We consider the 
 cell decompositions $\DD^n(h,\CC)$  as discussed in the
previous section. Note that the $1$-tiles (i.e., the tiles in
$\DD^1(h,\CC)$) are exactly the $1$-tiles in $\mathcal{S}^1$
under the identification $\mathcal{S}^1\cong
\mathcal{S}^0=S^2$ (see the bottom left in Figure~\ref{fig:1flap_both}). 
 
The map $h^n$ sends each $n$-tile to a
$0$-tile homeomorphically, and we can assign colors to $n$-tiles
so that $h^n$ sends an $n$-tile to the $0$-tile of the same
color. In the passage from $\DD^n(h,\CC)$ to $\DD^{n+1}(h,\CC)$
each $n$-tile is subdivided by $(n+1)$-tiles in the same way as
the $0$-tile of the same color is subdivided by $1$-tiles.  From
this it is clear that there is a one-to-one correspondence between
$n$-tiles in $\mathcal{S}^n$ and $n$-tiles for the pair $(h,\CC)$.
Moreover, these tiles realize identical combinatorics. More
precisely, 
\linebreak
we have
\begin{align*}
  \mathcal{X}^{n+1}\supsim \mathcal{X}^n 
  \quad &\Leftrightarrow \quad 
  X^{n+1} \supset X^n, \quad \text{and} 
  \\
  \mathcal{X}^n \cap \mathcal{Y}^n \neq \emptyset
  \quad &\Leftrightarrow \quad 
  {X}^n \cap {Y}^n \neq \emptyset,  
\end{align*}
where the $n$-tiles $\mathcal{X}^n$ and $\mathcal{Y}^n$ in $\mathcal{S}^n$ and the $(n+1)$-tile
$\mathcal{X}^{n+1}$ in $\mathcal{S}^{n+1}$
correspond to the $n$-tiles $X^n$ and $Y^n$ and the $(n+1)$-tile $X^{n+1}$
for $(h,\CC)$, respectively.  

Recall from Section~\ref{sec:cell-decomposition} that $h$ is
expanding if the diameters of $n$-tiles for $(h,\CC)$ tend to $0$
uniformly as $n\to \infty$ (with respect to some fixed base
metric on $S^2$ representing the topology).  In this case one obtains a
well-defined map $\varphi\: \mathcal{S}\ra S^2$ by sending a
point in $\mathcal{S}$ represented by a descending sequence
$\{\mathcal{X}^{n}\}$ to the unique point in the intersection
$\bigcap_{n\in \N_0} X^n$ of the corresponding $n$-tiles for
$(h,\CC)$.  In general, our map $h$ need not be expanding, but we
may assume this if we choose the identification of
$\mathcal{S}^1$ with $\mathcal{S}^0$ carefully (this hinges on
the fact that $h$ is ``combinatorially expanding'' and so the map
can be corrected if necessary to make it expanding; 
see Theorem~\ref{thm:combexp1} for details). 
It is then not hard
to see that $\varphi$ is a homeomorphism, and so $\mathcal{S}$ is
a $2$-sphere.

Though  $(\mathcal{S}, \varrho)$ is a topological
$2$-sphere, it is not a {\em quasisphere}. This means that
this space is not quasisymmetrically equivalent to the standard
$2$-sphere (i.e., the unit sphere in $\R^3$, or equivalently the
Riemann sphere $\CDach$ equipped with the chordal metric;   see
Section~\ref{sec:QCgeom} for the definition of a quasisymmetry). This 
is  closely related to the fact that $h$ is not
(equivalent to) a rational map. One can deduce that $(\mathcal{S}, \varrho)$ is not a quasisphere   from a general
result (see Theorem~\ref{thm:S2vsf}~\ref{item:S2qsphere} mentioned in
the next section), but one can also show this directly (we will
outline an argument in Section~\ref{sec:snowballs}).

The fractal sphere $(\mathcal{S}, \varrho)$ is in a sense   the natural domain for our map $h$; namely, we can conjugate our original map 
$h\: S^2\ra S^2$ by the homeomorphism $\varphi\: \mathcal{S} \ra S^2\ $ to obtain a map on $\mathcal{S}$, also  denoted
 by 
 $h$. 
 We also obtain $n$-tiles in $\mathcal{S}$ corresponding to the 
 $n$-tiles in $S^2$ under the homeomorphism $\varphi$. Roughly speaking, an $n$-tile  in  $\mathcal{S}$ is the part of $\mathcal{S}$
that ``sits on top'' of an $n$-tile in $\mathcal{S}^n$.  
The new map $h\: \mathcal{S}\ra \mathcal{S}$  then  behaves locally like a similarity map: it scales each  $(n+1)$-tile in $\mathcal{S}$ by a factor $2$ and matches it with  the corresponding $n$-tile.

\section{Visual metrics  and the visual sphere}
\label{sec:intvismet-vissph}

After this example we return to the general setting. Let $f\colon S^2\to S^2$ be an expanding Thurston map. We fix a
Jordan curve $\CC\subset S^2$ with $\post(f)\subset S^2$. Then for $n\in \N_0$ we have 
 cell decompositions $\DD^n=\DD^n(f,\CC)$ with $1$-skeleton $f^{-n}(\CC)$ as 
 defined in Section~\ref{sec:cell-decomposition}.

Since $f$ is expanding,  the diameters of $n$-tiles (i.e., tiles in $\DD^n$)
shrink to $0$ uniformly as $n\to \infty$. So if   $x,y\in S^2$ are   distinct points and 
$X^n$ and $Y^n$ are tiles of level $n$ 
with $x\in X^n$ and $y\in Y^n$, 
then $X^n\cap Y^n=\emptyset$ 
 for sufficiently large $n$. This implies that 
 the number\index{m@$m_{f,\CC}$}  
 \begin{align}
  \label{eq:defmx_intro2}
  m(x,y)\coloneqq \max\{n\in \N_0 : {} &\text{there exist
    non-disjoint $n$-tiles}
  \\
  \notag
  &X^n \text{ and } Y^n \text{ with } 
  x\in X^n, y\in Y^n\}   
\end{align}
is finite. 
Similar to
\eqref{eq:intromxy}, it  records the level at which $x$ and $y$ can
be separated by tiles. In Figure~\ref{fig:def_visual_metric}  we have illustrated  this
separation by  tiles in  an example. 
 
Generalizing \eqref{eq:def_visual_introsec}, we consider metrics
$\varrho$ on $S^2$ satisfying
\begin{equation*}
  \varrho(x,y) \asymp \Lambda^{-m(x,y)},
\end{equation*}
for some $\Lambda>1$. We call such a metric a 
\emph{visual metric}\index{visual metric}\index{metric!visual}\index{r@$\varrho$}
for $f$, and $\Lambda$ its {\em expansion factor}. We
will start investigating visual metrics in earnest in
Chapter~\ref{cha:visual-metrics}. 

Visual metrics for a given
Thurston map $f$ are not unique, but two different visual metrics
with the same expansion factor are bi-Lipschitz equivalent. 
They are snowflake equivalent if they have different expansion
factors (see Section~\ref{sec:QCgeom} for this
terminology). Whether a metric is visual does not depend on the
choice of the Jordan curve $\CC$ that was used to define  the quantity \eqref{eq:defmx_intro2} via the cell decompositions $\DD^n(f,\CC)$. Moreover, if $F=f^k$ is an iterate of $f$
(where $k\in \N$), then a metric is visual for $f$ if and only if it
is visual for $F$. These (and other) basic properties of visual metrics can be found 
in Proposition~\ref{prop:visualsummary}. 

If $\sigma$ is  a tile or an edge in the cell decomposition
$\DD^n= \DD^n(f,\CC)$,  then
$\diam_\varrho(\sigma)\asymp \Lambda^{-n}$. In addition, any two
disjoint 
cells $\sigma,\tau\in\DD^n$ satisfy
$\dist_\varrho(\sigma, \tau)\gtrsim \Lambda^{-n}$. This notation
means that there is a constant $C>0$ such that
$C\dist_\varrho(\sigma,\tau) \geq \Lambda^{-n}$. We refer to the
constant $C$ as $C(\gtrsim)$. Equivalently, we write
$\Lambda^{-n}\lesssim \dist_\varrho(\sigma,\tau)$ and refer to the
constant  $C$ as $C(\lesssim)$. Here the constants $C(\asymp)$ and
$C(\gtrsim)$ do not depend on $n$ or the cells involved. In fact,
these two geometric properties characterize visual
metrics (see Proposition~\ref{lem:expoexp}).

For the map $g$ from
Section~\ref{sec:Lattes} the length metric induced by the Euclidean metric on the pillow
$P$ is a visual metric  with expansion factor $\Lambda=2$. 
Similarly, the particular metric $\varrho$ defined in
Section~\ref{sec:int-frac-sph} is a visual metric for $h$ with
expansion factor $\Lambda=2$ (here we identify $\mathcal{S}$ with
$S^2$ by the homeomorphism $\varphi$).  In this case, we obtain
visual metrics with arbitrary expansion factor $1<\Lambda \leq 2$
if we consider a ``snowflaked'' metric $\varrho^\alpha$ 
with suitable $\alpha\in (0,1]$, 
but there is no visual metric for $h$
with $\Lambda >2$.  Indeed, if $\varrho$ is a visual metric with
expansion factor $\Lambda>1$ and $X^n$ is an $n$-tile, then
$\diam_{\varrho}(X^n)\lesssim \Lambda^{-n}$.  
Now it is easy to see
that one can form a connected chain of $n$-tiles with $2^n$
elements that joins two non-adjacent $0$-edges (as
Figure~\ref{fig:intpolysur} suggests, one obtains such a chain by
running along the bottom $0$-edge). Then by the triangle
inequality $2^n\cdot \Lambda^{-n}\gtrsim 1$ for all $n\in \N$,
and so $\Lambda\le 2$.

Let $f\: S^2\ra S^2$ be an expanding Thurston map. Then the
supremum of all $\Lambda>1$ for which there exist visual metrics
with expansion factor $\Lambda$ agrees with the
\emph{combinatorial expansion factor}\index{combinatorial expansion factor}\index{L0@$\Lambda_0$}  
of $f$, denoted by
$\Lambda_0(f)$.  It is computed from data associated with the
cell decompositions $\DD^n(f,\CC)$ determined by the map
$f$ and a  Jordan curve $\CC\sub S^2$   with
$\post(f)\sub \CC$.  For this we consider  the combinatorial quantity
$D_n(f,\CC)$  defined to be the minimal number of tiles in
$\DD^n(f,\CC)$ that are needed to form a connected set joining
opposite sides of $\CC$, 
i.e., 
two non-adjacent $0$-edges 
(the
definition is slightly different in the case $\#\post(f)=3$; see 
Section~\ref{sec:opp}). For the examples discussed in Sections~\ref{sec:Lattes} and~\ref{sec:int-frac-sph} we have 
 $D_n(g,\CC)=2^n$ and  $D_n(h,\CC)=2^n$ for $n\in \N_0$.

In general, the number $D_n(f,\CC)$ depends on $\CC$. For an
expanding Thurston map it grows at an exponential rate as
$n\to \infty$. This growth rate is independent of $\CC$, and 
determined only by $f$. Moreover, the limit 
\begin{equation}\label{intro:combexp} 
  \Lambda_0(f) \coloneqq \lim_{n\to \infty} D_n(f,\CC)^{1/n}
\end{equation} 
exists, satisfies $1<\Lambda_0(f) <\infty$, and is defined to be
the combinatorial expansion factor of $f$ (see
Proposition~\ref{prop:exp}).  It is invariant under
topological conjugacy  and well-behaved under iteration (see
Proposition~\ref{prop:expfacinv}). For our two examples we have 
$\Lambda_0(g) =2$ and $\Lambda_0(h)=2$.

As already mentioned
above, $\Lambda_0(f)$ gives the range of possible expansion
factors of visual metrics for an expanding Thurston map $f$. This is made precise 
in the following theorem. 

\begin{reptheorem}{thm:visexpfactors1} 
  [Visual metrics and their expansion factors] 
Let   $f\:S^2\ra S^2$ be  an expanding Thurston map, and 
  $\Lambda_0(f)\in (1,\infty)$ be its combinatorial expansion
  factor. Then the following statements are true:
  \begin{enumerate}
    
  \item
   If $\Lambda$ is the expansion factor of a visual metric for
   $f$, then $1<\Lambda\le \Lambda_0(f)$.
 
 \item
   Conversely, if $1<\Lambda<\Lambda_0(f)$, then there exists a
   visual metric $\varrho$ for $f$ with expansion factor
   $\Lambda$. Moreover, the visual metric $\varrho$ can be
   chosen to have the following additional property: 
   
   For every
   $x\in 
   S^2$ there exists a neighborhood $U_x$ of $x$ such that
   \begin{equation*} 
     \varrho(f(x), f(y))
     =\Lambda  \varrho(x,y)\text{ for all } y\in U_x.
   \end{equation*} 
 \end{enumerate}
\end{reptheorem} 
(Note that in this introduction we label the results as they
appear in later chapters.) 

In general, one cannot guarantee the existence of a visual metric
with expansion factor $\Lambda=\Lambda_0(f)$ (see
Example~\ref{ex:notattained}). 

The  combinatorial expansion factor
 always satisfies the inequality   $\Lambda_0(f) \leq
\deg(f)^{1/2}$, where $\deg(f)$ is the (topological) degree of $f$ (see Proposition~\ref{prop:macgrdn}). For our   examples $g$ and $h$ from the  previous sections we have 
 $\Lambda_0(g) = 2 =
\deg(g)^{1/2}$ and $\Lambda_0(h) = 2 <
\deg(h)^{1/2} = \sqrt{5}$. The equality for the Latt\`es map $g$ is not a coincidence.  Closely related results will be
discussed in Section~\ref{sec:char-latt-maps}. 

According to 
Theorem~\ref{thm:visexpfactors1}~\ref{item:visexpfactors2}, for each expanding Thurston map $f$ we can find a
visual metric $\varrho$ so that $f$ scales the metric $\varrho$
by a constant factor at each point. The
Latt\`{e}s map $g\:\CDach\ra \CDach$ discussed in
Section~\ref{sec:Lattes}  illustrates this statement: if we
equip $\CDach$ with a suitable visual metric for $g$ (the path
metric on the pillow in Figure~\ref{fig:mapg}), then $g$ behaves like a
piecewise similarity map, where distances are scaled by the
factor $\Lambda=2$.

The space
$\mathcal{S}$ from Section~\ref{sec:int-frac-sph} equipped with the visual metric $\varrho$ in
\eqref{eq:def_visual_1} is a fractal sphere.  It  is self-similar in the sense
that the part of the surface that is ``built on top'' of some
$n$-tile $\mathcal{X}^n$ is similar (i.e., is isometric up to
scaling by the factor $2^n$) to the part of the surface that is
``built on top'' of the white or the black $0$-tile. Similarly, we can find visual metrics for any Thurston map $f$ such
that $f$ scales tiles by a constant factor. Then  the 
metric behavior of the dynamics on tiles becomes very simple,
while the space on which $f$ acts is a fractal sphere and  geometrically more
complicated.

Our choice of the term ``visual metric'' is motivated by the
close relation of this concept to the notion of a visual metric
on the boundary of a Gromov hyperbolic space (see
Section~\ref{sec:Grhyp} for general background; very similar
ideas can be found in \cite{HP}). Namely, if $f\:S^2\ra S^2$ is
an expanding Thurston map and $\CC\sub S^2$ a Jordan curve with
$\post(f)\sub \CC$, then one can define an associated {\em tile
  graph} $\G(f,\CC)$ as follows. Its vertices are given by the
tiles in the cell decompositions $\DD^n(f,\CC)$ on all levels
$n\in \N_0$.  We consider $X^{-1}\coloneqq S^2$ as a tile of level $-1$ and add it as a vertex.  One joins two vertices by an edge if the corresponding
tiles intersect and have levels differing by at most $1$ (see
Chapter~\ref{cha:Gromov}). The graph $\G(f,\CC)$ depends on the
choice of $\CC$, but if $\CC'\sub S^2$ is another Jordan curve
with $\post(f)\sub \CC'$, then $\G(f,\CC)$ and $\G(f,\CC')$ are
rough-isometric (Theorem~\ref{thm:roughisom}); note that this is
much stronger than being quasi-isometric (see
Section~\ref{sec:Grhyp} for the terminology).

The graph $\G(f,\CC)$ is Gromov hyperbolic
(Theorem~\ref{thm:tileGrom}). Its boundary at infinity
$\partial_\infty\G(f,\CC)$ can be identified with $S^2$. 
Under this identification the class of visual metrics in the
sense of Gromov hyperbolic spaces coincides with the class of
visual metrics for $f$ in our sense
(Theorem~\ref{thm:visualGrTh}).  The number $m(x,y)$ defined in
\eqref{eq:defmx_intro2} is the Gromov product of the
 points $x,y\in S^2\cong \partial_\infty\G(f,\CC)$ up to a uniformly bounded  additive
constant (Lemma~\ref{lem:m_Gromov}).
  
If $f\: S^2\ra S^2$ is an expanding Thurston map and $\varrho$ a
visual metric for $f$, then we call the metric space
$(S^2,\varrho)$ the {\em visual sphere} of $f$. For fixed $f$  different visual metrics $ \varrho_1$ and $\varrho_2$ give snowflake equivalent spaces $(S^2,\varrho_1)$ and $(S^2,\varrho_1)$. So an expanding Thurston map determines its visual sphere uniquely up to snowflake equivalence. 
  
Many dynamical properties of $f$ are encoded in the geometry of
its visual sphere.  The following statement is one of the main
results of this work.

%

\begin{reptheorem}{thm:S2vsf}
  [Properties of $f$ and its associated visual sphere]
  Suppose $f\colon S^2\to S^2$ is an expanding Thurston map and $\varrho$ is a
  visual metric for $f$. Then the following statements are true:
  \begin{enumerate}

  \item 
    $(S^2,\varrho)$ is {doubling} if and only if $f$ has {no
      periodic critical points}. 
  
  \item 
    $(S^2,\varrho)$ is quasisymmetrically
    equivalent to $\CDach$ if and only if $f$ is
    topologically conjugate to a {rational map}.
  \item 
    $(S^2,\varrho)$ is {snowflake equivalent} 
    to $\CDach$ if and only if $f$ is topologically conjugate to a
    {Latt\`{e}s map}.  
  \end{enumerate}
\end{reptheorem}
Here it is understood that $\CDach$ is equipped with the chordal
metric.  For the terminology used in the statements
see Section~\ref{sec:QCgeom}.

As we already discussed, part~\ref{item:S2qsphere} of the previous theorem provides an
 analog of 
Cannon's conjecture\index{Cannon's conjecture} in geometric group
theory (see Section~\ref{sec:Cannconj} for a more detailed
discussion). According to this conjecture every Gromov hyperbolic
group $G$ whose boundary at infinity $\partial_\infty G$ is a
$2$-sphere should arise from some standard situation in
hyperbolic geometry. The conjecture is equivalent to showing that
$\partial_\infty G$ equipped with a visual metric (in the sense
of Gromov hyperbolic spaces) is quasisymmetrically equivalent to
$\CDach$. One of the reasons why Cannon's
conjecture is still open may be the lack of non-trivial examples
that guide the intuition (see the  paper \cite{BouK} though,
which in a sense addresses this issue). All examples come from
fundamental groups $G$ of compact hyperbolic manifolds
where one already has a natural identification of
$\partial_\infty G$ with $\CDach$; according to Cannon's
conjecture there are no other examples. In contrast, the visual
spheres of expanding Thurston maps provide a rich supply of
metric $2$-spheres that sometimes are and sometimes are not
quasisymmetrically equivalent to $\CDach$ (see
Section~\ref{sec:snowballs}).
  
The proof of one of the implications in
Theorem~\ref{thm:S2vsf}~\ref{item:S2qsphere} (the ``only if''
part) 
uses some well-known ingredients.  Namely, if $(S^2, \varrho)$ is
quasisymmetrically equivalent to the standard sphere $\CDach$,
then one can conjugate $f$ to a map $g$ on $\CDach$. Since the
map $f$ dilates distances with respect to a suitable visual
metric by a fixed factor (see
Theorem~\ref{thm:visexpfactors1}~\ref{item:visexpfactors2} mentioned above), the map
$g$ is uniformly quasiregular (see  Section~\ref{sec:QCgeom} for the terminology).  Hence $g$, and therefore also
$f$, are conjugate to a rational map by a standard theorem.
 
The converse direction (the ``if'' part) is harder to
establish. If $f$ is conjugate to a rational map, then we may
assume without loss of generality that $f$ is a rational expanding
 Thurston map on $\CDach$ to begin with. If $\varrho$ is
a visual metric for $f$, then one shows that the identity map
from $(\CDach, \varrho)$ to $(\CDach, \sigma)$ is a
quasisymmetry, where $\sigma$ is the chordal metric.  This
follows from a careful analysis of the geometry of the tiles in
the cell decompositions $\DD^n(f,\CC)$ with respect to the metric
$\sigma$ (see Proposition~\ref{prop:chordalm}). For example,
while it is fairly obvious from the definitions that adjacent
tiles in $\DD^n(f,\CC)$ have comparable diameter with respect to
a visual metric $\varrho$ (with uniform constants independent of
the level $n$), the same assertion is also true for the chordal
metric $\sigma$.  Our proof of this and related statements is
based on Koebe's distortion theorem and the fact that if $f$ has
no periodic critical points, then in the cell decompositions
$\DD^n(f,\CC)$ we see locally only finitely many different
combinatorial types.

Thurston studied the question when a given Thurston map is
represented by a conformal dynamical system from a point of view 
different from the one suggested by
Theorem~\ref{thm:S2vsf}~\ref{item:S2qsphere} (see
Section~\ref{sec:thurst-class-rati} for a short overview).  He
asked when a Thurston map $f\:S^2\ra S^2$ is in a suitable sense
 {\em (Thurston) equivalent} (see Definition~\ref{def:Thequiv}) to
a rational map and obtained a necessary and sufficient condition
(see 
\cite{DH}). For expanding Thurston maps his notion of equivalence
actually means the same as topological conjugacy of the maps
(Theorem~\ref{thm:exppromequiv}).

The proof of part~\ref{item:S2qsphere} of Theorem~\ref{thm:S2vsf} does not use Thurston's theorem. Indeed,  none of our statements relies on this, and so  our methods
possibly provide a different approach  for its proof.   

It is not clear how useful  Theorem~\ref{thm:S2vsf}~\ref{item:S2qsphere} 
is for deciding whether an explicitly  given expanding Thurston map is
topologically conjugate to  a rational map. 
It is likely that our techniques  can be used to formulate a more efficient criterion, but 
we will not pursue this further here.

\section{Invariant curves} 
\label{sec:int-inv-cell-de}

The Jordan curve $\CC$ chosen in
Section~\ref{sec:Lattes} is invariant for the map $g$ in the
sense that $g(\CC)\subset \CC$. In this case, the cell decomposition $\DD^{n+1}(g,\CC)$ is a
refinement of $\DD^n(g,\CC)$ for each $n\in \N_0$. We have a similar situation 
for the Jordan curve $\CC$ and the map $h$ in 
 Section~\ref{sec:int-frac-sph}. 

 Some of our main results are about the
existence and uniqueness of such invariant Jordan curves $\CC$.
In particular, we will show that they exist for    sufficiently high iterates  of  {\em every} expanding Thurston map.

\begin{reptheorem}{thm:main}
  [High iterates have invariant curves]
  Let $f\colon S^2\to S^2$ be an expanding Thurston map, and
  $\CC\sub S^2$ be a Jordan curve with $\post(f)\sub \CC$.  
  Then
  for each 
  sufficiently 
  large $n\in \N$ there exists a Jordan
  curve $\widetilde{\CC}\sub S^2$ that is invariant for $f^n$ and
  isotopic to $\CC$ rel.\ $\post(f)$.
\end{reptheorem}

A discussion of isotopies and related terminology can be found in
Section~\ref{sec:thurston-equivalence}. 
Since $\widetilde{\CC}$ is
isotopic to $\CC$ rel.\  $\post(f)$, it will also contain the set
$\post(f)$.

In Example~\ref{ex:noinvCC} we exhibit an  expanding
Thurston map $f\colon S^2\to S^2$ that has no $f$-invariant Jordan
curve $\widetilde{\CC}\subset S^2$ with
$\post(f)\subset \widetilde{\CC}$. This shows  that in general
it is necessary to pass to  an iterate in
Theorem~\ref{thm:main}.

If a curve $\widetilde{\CC}$ is invariant for some iterate $f^n$,
then one cannot expect it to be invariant for some other iterate
$f^k$ unless $k$ is a multiple of $n$ (see
Remark~\ref{rem:ndep}). So typically the curve
$\widetilde{\CC}$ in the previous theorem will depend on $n$.

The proof of Theorem~\ref{thm:main} is based on a necessary and
sufficient criterion for  the existence of $f$-invariant curves
given in Theorem~\ref{thm:exinvcurvef}. An outline of the proof
of this latter theorem is presented in Example~\ref{ex:invC}
(see  Figure~\ref{fig:invC_constr} for an illustration).  

One can actually formulate a related criterion for the existence
of an invariant curve in a given isotopy class rel.\ $\post(f)$
or rel.\ $f^{-1}(\post(f))$.
Moreover, if an $f$-invariant
Jordan curve $\widetilde \CC$ exists, then it is the Hausdorff
limit of a sequence of Jordan curves $\CC^n$ that can be
obtained from a simple iterative procedure
(see Remark~\ref{rem:expcheck}~(iii)) and Proposition~\ref{prop:invCit}).

Our existence results are complemented by the following
uniqueness statement for invariant Jordan curves.

\begin{reptheorem}{thm:uniqc}
  [Uniqueness of invariant curves]
  Let $f\: S^2\ra S^2$ be an expanding Thurston map, and 
  $\CC, \CC'\sub S^2 $ be $f$-invariant Jordan curves  that both
  contain the set $\post(f)$.  Then $\CC=\CC'$ if and only if
  $\CC$ and $\CC'$ are isotopic rel.\ $f^{-1}(\post(f))$.
\end{reptheorem}

As a consequence one can prove that if $\#\post(f)=3$, then there
are at most finitely many $f$-invariant Jordan curves
$\CC\sub S^2$ with $\post(f)\sub \CC$ 
(Corollary~\ref{cor:finitepost3}). This is also true if $f$ is
rational  and has a hyperbolic orbifold (see Theorem~\ref{thm:rat_finitelyC}).  
 In general, a
Thurston map $f$ can have infinitely many such invariant curves
(Example~\ref{ex:infty_C}), but there are at most finitely many
in a given isotopy class rel.\ $\post(f)$ 
(Corollary~\ref{cor:finiterelP}).

Let $f\colon S^2\to S^2$ be a Thurston map and suppose
$\CC\subset S^2$ is an $f$-invariant Jordan curve with
$\post(f)\subset \CC$. We consider the cell decompositions
$\DD^n=\DD^n(f,\CC)$ as discussed in
Section~\ref{sec:cell-decomposition}.  The $f$-invariance of
$\CC$ implies that $f^{n+1}(f^{-n}(\CC)) \subset \CC$, or
equivalently that $f^{-n}(\CC) \subset f^{-(n+1)}(\CC)$ for all
$n\in \N_0$. Since $n$-tiles and $(n+1)$-tiles were defined to be
the closures of the complementary components of $f^{-n}(\CC)$ and
$f^{-(n+1)}(\CC)$, respectively, each $(n+1)$-tile is contained
in an $n$-tile. Similarly, every cell in $\DD^{n+1}$ is contained
in a cell in $\DD^n$. More precisely, $\DD^{n+1}$ is a refinement
of $\DD^n$ (see Definition~\ref{def:ref} and
Proposition~\ref{prop:invmarkov}).  In particular, $\DD^1$
refines $\DD^0$.

One can essentially recover the  Thurston map $f$ from the 
cell decomposition $\DD^0$ and its refinement $\DD^1$ if one specifies  some  
additional data. In the example in Figure~\ref{fig:mapg} we
labeled the vertices in domain and range to indicate their correspondence under the map. Similarly, in the general case this additional
information is provided by a  \emph{labeling},
which is a map $L\colon \DD^1\to \DD^0$ (see
Section~\ref{sec:labelings}).

The triple $(\DD^1,\DD^0,L)$ records how  $0$-cells are subdivided by 
$1$-cells and how $1$-cells are mapped to $0$-cells. 
We call such triples $(\DD^1,\DD^0,L)$ 
{\em two-tile subdivision rules}\index{two-tile subdivision rule}\index{subdivision} 
(see Definition~\ref{def:subdivcomb}), because $\DD^0$ 
contains two tiles (namely the two closed Jordan regions bounded by $\CC$). 
Every Thurston map $f$ with an $f$-invariant Jordan curve 
$\CC\subset S^2$ with $\post(f) \subset \CC$ gives rise to a two-tile subdivision rule
(see Proposition~\ref{prop:ThmapSub}). 

Note that $L$ is a map between finite sets. This means that the
information encoded in $(\DD^1,\DD^0,L)$ is given in terms of
finite data. So  a two-tile subdivision rule can be considered as a 
combinatorial object. 

Conversely, every two-tile subdivision rule can be {\em realized}
by a Thurston map (see Proposition~\ref{prop:rulemapex}). It is unique up to 
Thurston equivalence. 
In this way,  two-tile subdivision rules give simple
combinatorial models for  Thurston maps.
There is no obvious difference between the two-tile
subdivision rules realized by rational maps and the ones that are
not. This is another motivation for  investigating  general Thurston maps.

The map $g$ in Section~\ref{sec:Lattes} and the map
 $h$ in 
Section~\ref{sec:int-frac-sph}  were constructed from 
Figure~\ref{fig:mapg} and Figure~\ref{fig:1flap_both}, respectively.  This really means that the pictures represent  two-tile subdivision rules, and the
maps  realize these subdivision rules according to
Proposition~\ref{prop:rulemapex}. This is our preferred
way to define Thurston maps. To be able to discuss 
examples, we will use this way of constructing
Thurston maps in an informal way even before we provide the theoretical foundations in Chapter~\ref{cha:subdivisions}. 

Our concept of a two-tile subdivision rule is closely related to   the general \emph{subdivision rules} that have been studied extensively by
Cannon, Floyd, and Parry (see for example \cite{CFP01}). 

The main consequence of Theorem~\ref{thm:main} is that we have a
combinatorial description by a two-tile subdivision rule for
sufficiently high iterates $F=f^n$ of every expanding Thurston
map $f$.

\begin{repcor}{cor:subdivnlarge} 
  [Thurston maps and subdivision rules]
  Let $f\: S^2 \ra S^2$ be an expanding Thurston map.  
  Then for
  each sufficiently large $n\in \N$ there exists a two-tile subdivision
  rule that is realized by $F=f^n$.
\end{repcor} 

In particular, we
obtain a cellular Markov partition for $F$. There are several other
approaches to  providing combinatorial
models for   certain classes of maps. For example, a postcritically-finite polynomial can  be
described by its Hubbard tree (see \cite{DH84}) or a rational map
with three critical values  by a dessin d'enfant (see 
\cite{Gro97} and \cite{LanZvo}). 
 A very general
setting that allows one to address similar questions is the
recently developed theory of self-similar group actions. In this context one investigates 
 algebraic 
objects such as 
the iterated monodromy group and the biset (or
bimodule) defined for a Thurston map (see \cite{Ne}, in particular Chapter~6). 

Our approach is more geometric and adapted to Thurston maps. One
of its main features 
is that we have  good geometric control for the cells in the 
decompositions $\DD^n=\DD^n(f,\CC)$ if $\CC$ is
$f$-invariant. In particular,  with respect to any visual metric the curve $ \CC$ is
actually a quasicircle 
(Theorem~\ref{thm:Cquasicircle}) and   
 the boundaries of the
tiles in $\DD^n$ are quasicircles with uniform parameters
independent of the level 
$n$ (Proposition~\ref{prop:arc}). 
For 
a rational expanding Thurston map  $f\colon \CDach \to \CDach$  the  tiles in $\DD^n$ are in
fact uniform quasidisks with respect to the chordal
metric $\sigma$ on $\CDach$ (Theorem~\ref{thm:main01}~\ref{item:frat_Markov}).

\section{Miscellaneous results} \label{sec:infures} 

In this section we collect various noteworthy results that may be  useful for the orientation of the reader.

%
%

The concept of {\em Thurston equivalence} for Thurston maps  was already mentioned before. We record its slightly technical definition in Section~\ref{sec:thurston-equivalence}. At first sight the concept does not seem to be adapted to the 
dynamics under iteration. However,  in Theorem~\ref{thm:exppromequiv} we will  
prove the important fact that two expanding Thurston maps are Thurston
equivalent if and only if they are topologically conjugate.

In Chapter~\ref{cha:symdym} we make a brief excursion to symbolic dynamics. The properties of visual metrics are essential for  proving the
following statement (see Chapter~\ref{cha:symdym} for the
relevant definitions).
 
\begin{reptheorem}{thm:expThfactor}
  Let $f\: S^2\ra S^2$ be an expanding Thurston map. Then $f$ is
  a factor of the left-shift $\Sigma\: J^\om\ra J^\om $ on the
  space $J^\om$ of all sequences in a finite set $J$ of
  cardinality $\#J=\deg(f)$.
\end{reptheorem}

The proof of this theorem does not use  invariant
Jordan curves for $f$ or its iterates;  so in a sense  it is independent of
Theorem~\ref{thm:main} mentioned above. It can be used  
to obtain another  Markov partition for $f$, but we have very little control for  the
geometric shape of the ``tiles''.

An immediate consequence of this theorem and its proof is the
fact that the periodic points of an expanding Thurston map
$f\: S^2\ra S^2$ form a  dense subset of $S^2$
(Corollary~\ref{cor:perdense}).

%

In Chapter~\ref{cha:quotiens} we investigate equivalence
relations $\sim$ on the sphere $S^2$, and the question when a
Thurston map $f\colon S^2\to S^2$ descends to a Thurston map on the quotient space 
$S^2/\Sim$. Here we assume that  $\sim$ is of  \emph{Moore-type}
(see Definition~\ref{def:eq_Moore_type}), which 
implies that 
$S^2/\Sim$ is a $2$-sphere. 
Under this assumption the relevant condition is that  $\sim$ is \emph{strongly invariant} for $f$ in the sense that $f$ maps each 
equivalence class onto  another equivalence class (see
Definition~\ref{def:sim_strongly-inv} and
Lemma~\ref{lem:eq_strong_inv}). We will prove that  for 
a given  equivalence relation $\sim$ of {Moore-type} on
$S^2$  a Thurston map $f\: S^2\ra S^2$
descends to a Thurston map if and only if $\sim$ is strongly
invariant for $f$  (see
Theorem~\ref{thm:f_descends_branched_cover} and
Corollary~\ref{cor:f_descends_Thurston}).

%
%

Often it is desirable to promote a given Thurston map $f$ that is
not expanding to an expanding one. More precisely, we want to find an expanding Thurston map $\widetilde{f}$
that is Thurston equivalent to $f$. In general, the existence of 
$\widetilde{f}$ is not guaranteed. However, if $f$ is
\emph{combinatorially expanding} (see
Definition~\ref{def:combexp}) such a map $\widetilde{f}$ does
exist. Roughly speaking, it is constructed by defining an
equivalence relation that collapses the sets where $f$ fails to
be expanding to points (see Chapter~\ref{cha:combexp}).

We will also investigate some measure-theoretic aspects of
expanding Thurston maps. Each such map has a natural measure
adapted to its dynamics.

\begin{reptheorem}{thm:maxentr0}
Let $f\: S^2 \ra S^2$ be an expanding Thurston map. Then there exists a unique measure $\nu_f$ 
of maximal entropy for $f$. The map $f$ is mixing for $\nu_f$. 
\end{reptheorem}

This theorem follows from results due  to Ha\"\i ssinsky-Pilgrim \cite[Theorem~3.4.1]{HP}.  We will present  a different proof 
and give an explicit description of $\nu_f$ in terms of the 
cell decompositions $\DD^n(F,\CC)$, where $F=f^n$ is a suitable iterate and $\CC$ is an invariant curve as in Theorem~\ref{thm:main}. 
In particular,  $\nu_f=\nu_F$ assigns equal mass to all tiles in the cell decompositions $\DD^n(F,\CC)$ of a given ``color'' (see Proposition~\ref{prop:exmeasure} and Theorem~\ref{thm:nuF}). 

The measure $\nu_f$ can be used to  study   the topological and 
measure-theoretic dynamics of $f$ under iteration. For example, we will see that  $h_{top}(f)=\log(\deg (f))$, where $h_{top}(f)$ is the topological entropy and $\deg (f)$ the topological degree of $f$ (Corollary~\ref{cor:topent}).

If $\mu$ is   a Borel  measure  on a metric space $(X,d)$, then we call the metric measure space $(X,d, \mu)$   {\em Ahlfors $Q$-regular}\index{Ahlfors regular} for $Q>0$  if  $$\mu(\overline B_d(x,r))\asymp r^Q$$  for each closed ball 
      $\overline B_d(x,r)$ in $X$ whose radius $r$ does not exceed the diameter of the space.  
If  an expanding Thurston map  has no periodic critical points, then   its  visual sphere together with its measure of maximal entropy  has this property.

\begin{repprop}{prop:Ahlforsreg} 
  Let $f\: S^2\ra S^2$ be an expanding
  Thurston map without periodic critical points, $\varrho$ be a visual
  metric for $f$ with expansion factor $\Lambda>1$, and  $\nu_f$ be 
  the measure of maximal entropy of $f$.
   Then   the metric measure space    $(S^2, \varrho, \nu_f)$ is Ahlfors $Q$-regular
    with
    \begin{equation*}
      Q\coloneqq \frac{\log(\deg(f))}{\log(\Lambda)}.
    \end{equation*}
 In particular, $(S^2,\varrho)$ has Hausdorff dimension $Q$ and 
 \begin{equation*}
   0<\mathcal{H}_\varrho^Q(S^2)<\infty.
 \end{equation*}
 \end{repprop}
Here $\mathcal{H}_\varrho^Q$ denotes Hausdorff $Q$-measure 
on the metric space $(S^2, \varrho)$. 

In Chapter~\ref{cha:rati-thurst-maps-1} we delve into a deeper analysis of measure-theoretic properties of rational 
expanding Thurston maps. 
We denote by $\leb_{\CDach}$ Lebesgue measure on $\CDach$,
normalized such that $\leb_{\CDach}(\CDach)=1$. 
Then the following (well-known) statement is true. 

\begin{reptheorem}{thm:ergodic_for_f}
  Let $f\colon \CDach\to \CDach$ be a rational expanding
  Thurston map. Then Lebesgue measure $\leb_{\CDach}$ is ergodic
  for $f$. 
\end{reptheorem}

Note that $\leb_{\CDach}$ is essentially never $f$-invariant, but  ergodicity is interpreted as for $f$-invariant measures (see the discussion in Section~\ref{sec:reviewmdyn}): if $A\sub \CDach$ is a Borel set with $f^{-1}(A)=A$, then 
 $\leb_{\CDach}(A)=0$ or  
 $\leb_{\CDach}(A)=1$.

In our context one can actually find an $f$-invariant measure that is absolutely continuous with respect to Lebesgue measure.

\begin{reptheorem}{thm:ex_inv_abs_L}
  Let $f\colon \CDach\to \CDach$ be a rational expanding Thurston
  map. Then there exists a unique $f$-invariant (Borel) probability
  measure $\lambda_f$ on $\CDach$ that is absolutely continuous
  with respect to Lebesgue measure $\leb_{\CDach}$. This measure
  has the form $d\lambda_f=\rho\, d\leb_{\CDach}$, where $\rho$
  is a positive continuous function on
      $\CDach\setminus \post(f)$. Moreover, the measure $\lambda_f$
  is ergodic for $f$.
\end{reptheorem}
Again this statement is essentially well known. We will 
prove it by interpreting the existence of $\lambda_f$ as a fixed point problem for a suitable  {\em Ruelle operator}. This is a standard technique in ergodic theory reviewed in Chapter~\ref{cha:rati-thurst-maps-1}.

\section{Characterizations of Latt\`{e}s maps}
\label{sec:char-latt-maps}

Latt\`{e}s maps form another major theme  in this
book. One  may define such a map as a rational expanding Thurston
map with a parabolic orbifold (see
Section~\ref{sec:orbif-assoc-thurst} for the
terminology). Equivalently, they are characterized as quotients of
holomorphic torus endomorphisms, or as quotients of holomorphic automorphisms on the complex plane $\C$ by a crystallographic
group (see Theorem~\ref{thm:Lattesstruc}). While these latter
descriptions are more technical to state, they contain more
information and allow us to construct all Latt\`{e}s maps
explicitly.

In Section~\ref{sec:Lattes} the Latt\`{e}s map $g$ was
constructed from maps $A\colon \C\to \C$ and
$\Theta\colon \C\to \CDach$. For these maps we have 
$g\circ \Theta = \Theta \circ A$, and so we obtain  a commutative diagram as in 
\eqref{eq:Lattes}. The push-forward of the Euclidean
metric on $\C$ by $\Theta$ is the 
{\em canonical orbifold metric}\index{canonical orbifold!metric}\index{metric!canonical orbifold}\index{o@$\omega$}\index{orbifold!canonical metric}\index{push-forward!of metric!by orbifold covering map}
$\omega$ of $g$.  In this example, it  is the path metric on the pillow in
Figure~\ref{fig:mapg}. Moreover, $\omega$ is a visual
metric for $g$. This is characteristic for Latt\`{e}s maps.

 To make this
 precise,  we will introduce some terminology in an
informal way.
Each  rational Thurston map $f\colon \CDach \to \CDach$ has
an associated orbifold $\OC_f$ (see
Section~\ref{sec:orbif-assoc-thurst}), for which there is in turn
a universal orbifold covering map $\Theta\colon X\to \CDach$ (see
Section~\ref{sec:orbifolds-coverings}). Here $X=\C$ or $X= \D$
depending on whether $f$ has a parabolic or hyperbolic
orbifold. The map $\Theta$ is holomorphic. The {\em canonical orbifold
metric} $\omega$ of $f$ is the push-forward of the Euclidean
metric (if $X=\C$) or hyperbolic metric (if $X=\D$) by $\Theta$ (see Section~\ref{sec:expratThmaps}). The
metric $\omega$ is a conformal metric on $\CDach$ (see Section~\ref{sec:metrspterm} for the terminology), 
and  closely related to the  chordal metric
$\sigma$ on $\CDach$. In fact,  
if  $f$ does not have periodic critical points, the metric space 
$(\CDach,\omega)$ is bi-Lipschitz equivalent to
$(\CDach, \sigma)$ (see Lemma~\ref{lem:om_chordal}). Note that all  this is tied to a  holomorphic setting, and so we cannot define such a canonical metric $\omega$
unless the Thurston map $f$ is rational. 

\begin{repprop}{prop:orbivispara}  
  [Canonical orbifold metric as visual metric] 
  Let $f\: \CDach\ra \CDach$ be a rational Thurston map without
  periodic critical points, and $\om$ be the canonical orbifold
  metric for $f$. Then $\om$ is a visual metric for $f$ if and
  only if $f$ is a Latt\`{e}s map.
\end{repprop}

In Theorem~\ref{thm:S2vsf}~\ref{item:S2Lattes} we have already
encountered  another (much deeper) characterization of Latt\`{e}s maps in
terms of visual metrics. 

It is not hard to see that for a Latt\`es map $f$ the space $(\CDach,\omega)$ is Ahlfors $2$-regular.
The existence of a visual metric with this property again characterizes Latt\`{e}s maps.

\begin{reptheorem}{thm:visual_2Ahlfors_Lattes}
  Let $f\colon S^2 \to S^2$ be an expanding Thurston map. Then
  $f$ is topologically conjugate to a Latt\`{e}s map if and only
  if there is a visual metric $\varrho$ for $f$ such that
  $(S^2, \varrho)$ is Ahlfors $2$-regular.
\end{reptheorem}

Together with
Proposition~\ref{prop:Ahlforsreg} (mentioned
above) the previous theorem implies  that an expanding Thurston map without periodic critical
points is topologically conjugate to a Latt\`{e}s map if and only
if there is a visual metric with expansion factor
$\Lambda=\deg(f)^{1/2}$ (see Corollary~\ref{cor:Lattes_L_degf}).

Recall from Theorem~\ref{thm:visexpfactors1} that for an
expanding Thurston map $f$ the supremum of all expansion factors
of visual metrics is given by the combinatorial expansion factor
$\Lambda_0(f)$. This number was defined in \eqref{intro:combexp}
as the limit of $D_n(f,\CC)^{1/n}$ as $n\to \infty$. Here
$D_n(f,\CC)$ is the minimal number of $n$-tiles that are needed
to form a connected set joining opposite sides of $\CC$.  We will
show in Proposition~\ref{prop:macgrdn} that
$D_n(f,\CC)\lesssim \deg (f)^{n/2}$.  Latt\`{e}s maps are
precisely those expanding Thurston maps for which this maximal
growth rate for $D_n(f,\CC)$ is attained.

\begin{reptheorem}{thm:Qianthm}
  Let $f\: S^2\ra S^2$ be an expanding Thurston map.  Then $f$
  is topologically conjugate to a Latt\`es map if and only if
  the following conditions are true:
 
  \begin{enumerate}
  \item  
    $f$ has no periodic critical points.
 
  \item   
    There exists $c>0$,  and a Jordan curve  $\CC\sub S^2$ 
  with $\post(f)\sub \CC$ such that for  all $n\in \N_0$ we have 
  \begin{equation*} 
  D_n(f,\CC)\ge c \deg (f)^{n/2}.
  \end{equation*} 
  \end{enumerate} 
  \end{reptheorem}
  This theorem is due to Qian Yin \cite{Qian15}.
It is remarkable that one can  characterize the conformal dynamical systems given by  iteration of Latt\`es maps   in terms of  essentially combinatorial data. 

An immediate consequence of the maximal growth rate of $D_n(f,\CC)$ is the inequality 
$\Lambda_0(f)\le \deg (f)^{1/2}$ for each expanding Thurston map. Here  equality is attained for Latt\`es
maps, and so one might expect that this property again
characterizes Latt\`es maps similar to Theorem~\ref{thm:Qianthm}.  However, there are   Thurston maps $f$ 
 satisfying
$\Lambda_0(f)= \deg(f)^{1/2}$ that are not topologically conjugate
to a Latt\`{e}s map (see Example~\ref{ex:notattained}).

The proof of Theorem~\ref{thm:Qianthm} relies on
Theorem~\ref{thm:visual_2Ahlfors_Lattes}, which in turn depends
on yet another characterization of Latt\`es maps.  For this we 
compare the Lebesgue measure
$\leb_{\CDach}$ on $\CDach$ with  the measure of maximal entropy
$\nu_f$ for a given rational expanding Thurston map $f\colon \CDach \to \CDach$. 
 As the
following result shows,  $\nu_f$ and $\leb_{\CDach}$  lie in different measure classes 
 unless $f$ is a Latt\`{e}s map. 
  
\begin{reptheorem}
  {thm:abscontimpLattes} 
  Let $f\: \CDach \ra \CDach$ be a rational expanding Thurston
  map. Then its measure of maximal entropy $\nu_f$ is
  absolutely continuous with respect to Lebesgue measure
  $\leb_{\CDach}$ if and only if $f$ is a Latt\`es map.
\end{reptheorem} 
  
This is a special case of a more general theorem due to 
Zdunik \cite{Zd} which gives a similar characterization of
Latt\`es maps among all, not only postcritically-finite, rational
maps.  
Crucial in the proof of 
Theorem~\ref{thm:abscontimpLattes} is the existence of the 
$f$-invariant measure $\lambda_f$ on $\CDach$ that is absolutely
continuous with respect to Lebesgue measure $\leb_{\CDach}$ (see
Theorem~\ref{thm:ex_inv_abs_L} mentioned above). 
In fact, for a Latt\`{e}s map $f$
the measures $\lambda_f$, $\nu_f$, 
and the {\em canonical orbifold measure} $\Omega_f$ (see Section~\ref{sec:expratThmaps}) all agree 
(Theorem~\ref{thm:Lattes_can_orb_meas}). In the example $g$ from
Section~\ref{sec:Lattes} these measures are given by the Euclidean
area measure on the pillow in Figure~\ref{fig:mapg} (normalized
to be a probability measure).

\section{Outline of the presentation} 
\label{sec:outline-presentation}
  
Our work is an introduction to the subject. We hope that it will
stimulate  more research in the area and will serve as a foundation for future
investigations. Therefore, we kept our presentation elementary,
as self-contained as possible, and rather detailed.

For the most part, the prerequisites for the reader are modest and include some
basic knowledge of complex analysis and topology, in particular
plane topology and the topology of surfaces.  A background in
complex dynamics is helpful, but not absolutely necessary. In
later chapters our demands on the reader are more
substantial. In particular, in Chapter~\ref{cha:measure} and
Chapter~\ref{cha:rati-thurst-maps-1} we require some concepts
and results from topological and measure-theoretic dynamics, but
we will state and review the necessary facts.

When  writing this  book, we were   faced with 
conflicting objectives. On the one hand, it  was desirable to present the material in linear  order, meaning that we  should only use results in an argument 
 that have been discussed or established before. 
On the other hand, a too rigid implementation of this idea would  have resulted in long detours that might have distracted  the reader from the subject
at hand. Moreover,  we often had to invoke results that  are
``well known'',  but difficult to find in the required form in the literature. For this reason we have included an
appendix, where such results are collected. We use and refer to
the appendix throughout the text. 

We will now give an outline of how of this book is organized and
will briefly discuss some main concepts and ideas.

In the last section of this introduction the reader can
find a list of all examples of Thurston maps that we consider in
this work (Section~\ref{sec:examples}).

In Chapter~\ref{cha:thmaps} we turn to Thurston maps, the main
object of our investigation. We first review branched covering
maps in Section~\ref{sec:branched-coverings} and then define
Thurston maps in Section~\ref{sec:defin-thurst-maps}.  Our notion
of \emph{expansion} is introduced in
Section~\ref{sec:expansion-1} (see Definition~\ref{def:exp}). In
this section we also give a characterization when a rational
Thurston map is expanding.  While the concept of expansion is
discussed more systematically later (in
Chapter~\ref{cha:expansion}), readers familiar with complex
dynamics may find it helpful to get some perspective early
on. The notion of \emph{(Thurston) equivalence} for Thurston maps
is discussed in Section~\ref{sec:thurston-equivalence}.
  
Every Thurston map $f$ has  an associated \emph{orbifold}
$\mathcal{O}_f$, defined in terms of the {\em ramification
  function} of $f$. This orbifold can be parabolic in some
exceptional cases (including Latt\`es maps) and is hyperbolic
otherwise (see Section~\ref{sec:orbif-assoc-thurst}). For
rational Thurston maps the \emph{universal orbifold covering
  map} induces a natural metric (the \emph{canonical orbifold
  metric}) and a natural measure (the \emph{canonical
  orbifold measure}). As we did not want to overburden  the reader with
technicalities at this early stage, we delegated a  detailed
discussion of these topics  to the appendix 
(see Sections~\ref{sec:orbifolds-coverings} and~\ref{sec:expratThmaps}).
 
As already mentioned, Thurston gave a characterization 
when a Thurston map is 
equivalent to a
rational map. We will review this without proofs in
Section~\ref{sec:thurst-class-rati}.  This material is not
really essential for the rest of the book, but we included this
discussion for general background.
 
In Chapter~\ref{cha:lattes-lattes-type} we discuss a large class
of Thurston maps, namely Latt\`es maps and a related class that
we call Latt\`es-type maps. Latt\`es-type maps are quotients of
torus endomorphisms, but in contrast to the Latt\`es case we do
not require the endomorphism to be  holomorphic. We will classify
Latt\`{e}s maps, which is surprisingly involved, and will discuss many examples.

In Chapter~\ref{cha:QCRoughGeo} we collect  facts from quasiconformal geometry (Section~\ref{sec:QCgeom}) and from the theory of Gromov hyperbolic spaces (Section~\ref{sec:Grhyp}) that will be relevant later on. 
We then turn to Cannon's conjecture in geometric group theory (Section~\ref{sec:Cannconj}). As we already remarked 
earlier, this conjecture  gives an intriguing analog to some of the main themes of our study of expanding Thurston maps. We illustrate this with an explicit description of some examples of fractal $2$-spheres that arise as visual spheres of expanding Thurston maps (Section~\ref{sec:snowballs}). We included Sections~\ref{sec:Cannconj} and~\ref{sec:snowballs}  mainly  to give some motivating background  for our investigation. 

This starts in earnest in Chapter~\ref{cha:celldecomp}, the
technical core of our combinatorial approach. Here we discuss
cell decompositions and their relation to Thurston maps. In
Section~\ref{s:celldecomp} we collect some general (well-known)
facts about cell decompositions, including the definition of a cell decomposition and  related concepts such as refinements and cellular maps. 
In Section~\ref{sec:2spherecd} we specialize to cell
decompositions on $2$-spheres.

In Section~\ref{sec:tiles} we consider cell decompositions
induced by a Thurston map $f$. Here we define cell decompositions
$\DD^n=\DD^n(f,\CC)$ for each level $n\in\N_0$ from a Jordan curve
$\CC\subset S^2$ with $\post(f) \subset \CC$ as outlined in
Section~\ref{sec:cell-decomposition}. These cell decompositions
are our most important technical tool for studying Thurston
maps. Their properties are summarized in
Proposition~\ref{prop:celldecomp}.

Given such a sequence $\DD^n$ of cell decompositions induced by
(the iterates of) a Thurston map $f$, we may \emph{label} the
cells in $\DD^n$  to record to which cells in
$\DD^0$ they are mapped by $f^n$. This is explained  in Section~\ref{sec:labelings}. By 
using two cell
decompositions $\DD^0$ and $\DD^1$ and a labeling (satisfying some additional assumptions), it is possible to construct a
Thurston map $f$ that {\em realizes} this  data in a
suitable way (see Proposition~\ref{prop:thurstonex}). Roughly
speaking,  this means that we may construct Thurston maps in a
geometric fashion, very similar to the example  indicated in
Figure~\ref{fig:mapg}.

In Section~\ref{sec:flowers} we introduce the concept of an {\em
  $n$-flower} $W^n(p)$ of a vertex $p$ in the cell decomposition
$\DD^n$. The set $W^n(p)$ is formed
by the interiors of all cells in $\DD^n$ that meet $p$ (see
Definition~\ref{def:flower} and Lemma~\ref{lem:flowerprop}). An
important fact is that while in general a component of the
preimage $f^{-n}(K)$ of a small connected set $K$ will not be
contained in an $n$-tile (i.e., a $2$-dimensional cell in
$\DD^n$), it is always contained in an $n$-flower
(Lemma~\ref{lem:preimsmall}).

 In Section~\ref{sec:opp} we give a precise definition  for  a 
connected set to {\em join opposite sides} 
of $\CC$.
In addition, we define the quantity
$D_n=D_n(f,\CC)$ that measures the combinatorial expansion rate
of a Thurston map. It is given as the minimal number of $n$-tiles
needed to form a connected set joining opposite sides of $\CC$
(see \eqref{def:dk}). 

In Chapter~\ref{cha:expansion} we revisit our notion of
expansion. The main result here is
Proposition~\ref{prop:expequivexp}, which gives several
equivalent conditions for a Thurston map to be expanding. In
particular,  it follows that expansion is a topological condition,
and does not depend on the choice of the metric on $S^2$ used in the
definition. Section~\ref{sec:further-results} collects some
additional  results about expansion, and  Section~\ref{sec:Latttypeexp}
provides a simple criterion when a Latt\`{e}s-type map is
expanding. 
 
In Chapter~\ref{cha:thurston-maps-23} we consider Thurston maps
with two or three  postcritical points. Every such map is
equivalent to a rational map. In
fact, every Thurston map $f$ with $\#\post(f)=2$ is equivalent
to $z\mapsto z^n$ for $n\in \Z\setminus\{-1,0,1\}$ 
 (see Theorem~\ref{thm:3postrat} and Proposition~\ref{prop:post2}). In
Section~\ref{sec:parab-thurst-polyn} we consider Thurston maps
with an  associated parabolic orbifold  of signature $(\infty,\infty)$
or $(2,2,\infty)$. 
 Such a map is equivalent to $z\mapsto z^n$ in the
case $(\infty,\infty)$, or to a \emph{Chebyshev polynomial} in
the case $(2,2,\infty)$ (up to a sign). This completes  the
classification of Thurston maps with parabolic orbifold begun in
Chapter~\ref{cha:lattes-lattes-type}.

In Chapter~\ref{cha:visual-metrics} we introduce \emph{visual
  metrics} for expanding Thurston maps, one of our
central concepts, as outlined in
Section~\ref{sec:intvismet-vissph}.  Basic properties of visual metrics are listed in
Proposition~\ref{prop:visualsummary}. A characterization (already
mentioned in Section~\ref{sec:intvismet-vissph}) is given in
Proposition~\ref{lem:expoexp}.  

If an expanding Thurston map $f$ is a rational map on the Riemann sphere $\CDach$, then   one would like to know whether some natural metrics on $\CDach$ are  visual metrics for $f$. 
The chordal metric on $\CDach$ is
never a visual metric. The canonical orbifold metric of $f$ is visual if and
only if $f$ is a Latt\`{e}s map. This is discussed in
Section~\ref{sec:orbifold_visual}.

In Chapter~\ref{cha:symdym} we briefly  turn to  symbolic
dynamics and show that 
every expanding Thurston map is a factor 
of a shift operator (see 
Theorem~\ref{thm:expThfactor}). 

In Chapter~\ref{cha:Gromov} we connect our development of
expanding Thurston maps to the theory of \emph{Gromov hyperbolic
spaces}. We define the \emph{tile graph}
$\mathcal{G}=\mathcal{G}(f,\CC)$ associated with an expanding
Thurston map $f\: S^2\ra S^2$ and a Jordan curve $\CC\sub S^2$
with $\post(f)\sub S^2$. The graph $\mathcal{G}(f,\CC)$ only
depends on $\CC$ up to rough-isometry
(Theorem~\ref{thm:roughisom}). We show that it is Gromov
hyperbolic (Theorem~\ref{thm:tileGrom}) and that its boundary at
infinity $\partial_\infty\mathcal{G}$ can be identified with
$S^2$ (Theorem~\ref{thm:visualGrTh}). Under this identification
a metric $\varrho$ on $S^2\cong \partial_\infty\mathcal{G}$ is 
visual in the sense of Gromov hyperbolic spaces 
if and only if it is visual for $f$ as defined in
Chapter~\ref{cha:visual-metrics} (see
Theorem~\ref{thm:visualGrTh}). This is the reason why we chose 
the  term ``visual'' for the metrics  associated with a Thurston map. 

In Chapter~\ref{cha:iso} we consider isotopies on $S^2$
and their lifts by Thurston maps.  We show that if two expanding
Thurston maps are Thurston equivalent, then they are in fact
topologically conjugate (Theorem~\ref{thm:exppromequiv}).
We also prove some results on isotopies of Jordan curves
(Section~\ref{sec:isotopy-rel.-postf}). The subsequent
Section~\ref{sec:graphs} contains some auxiliary statements on
graphs.  The main result is the important, but rather technical
Lemma~\ref{lem:isorelP},  which gives a sufficient criterion
when a Jordan curve can be isotoped into the $1$-skeleton of a
given cell decomposition of a $2$-sphere.

In Chapter~\ref{cha:subdivisions} we study the cell
decompositions $\DD^n=\DD^n(f,\CC)$ under the additional
assumption that $\CC$ is $f$-invariant. Then $\DD^{n+k}$
is a refinement of $\DD^n$ for all $n,k\in \N_0$; so each cell in any of the cell decompositions
$\DD^n$ is ``subdivided'' 
by cells of higher levels. 
Moreover,
the pair $(\DD^{n+k}, \DD^n)$ is a cellular Markov partition for
$f^k$ (Proposition~\ref{prop:invmarkov}). In this case the
Thurston map $f$ can be described by a {\em two-tile subdivision
  rule} (Definition~\ref{def:subdivcomb}) as discussed in
Section~\ref{sec:subdivisions}. Conversely, we may construct a
Thurston map from a two-tile subdivision rule by
Proposition~\ref{prop:rulemapex}. This is the main result in
this chapter. This way to construct Thurston maps from a
combinatorial viewpoint is illustrated in
Section~\ref{sec:examples-two-tile}, where we consider many
examples of Thurston maps given in this form.


In Chapter~\ref{cha:quotiens} we study  the question when a
Thurston map $f\: S^2\ra S^2$ descends to a Thurston map on the quotient space
$S^2/\Sim$ obtained from an equivalence relation $\sim$ on $S^2$. For this, we first review closed equivalence
relations and Moore's theorem in Section~\ref{sec:clos-equiv-relat}. We also require 
some auxiliary results on the
mapping behavior of branched covering maps discussed in
Section~\ref{sec:toppropbrcov}. We  prove
  the main result of this chapter in Section~\ref{sec:quot-thurst-maps}:   a Thurston map $f\: S^2\ra S^2$ descends to a Thurston
map on the quotient $S^2/\Sim$ obtained from an equivalence relation $\sim$ on $S^2$ of Moore-type
(see Definition~\ref{def:eq_Moore_type}) if and only if $\sim$ is
strongly $f$-invariant (see
Definition~\ref{def:sim_strongly-inv},
Theorem~\ref{thm:f_descends_branched_cover}, and
Corollary~\ref{cor:f_descends_Thurston}). 

Can a given  two-tile subdivision rule  be realized by an {\em expanding} Thurston map? 
This question is addressed in  Chapter~\ref{cha:combexp}. A
necessary condition is that the subdivision rule is
\emph{combinatorially expanding} (see
Definition~\ref{def:combexprule} and
Defi\-nition~\ref{def:combexp}).  We show 
that every combinatorially expanding Thurston map
is equivalent to an expanding Thurs\-ton map
(Proposition~\ref{prop:combexp} and Theorem~\ref{thm:combexp1}).

To prove this statement, we ``correct'' a Thurston map $f\: S^2\ra S^2$  that
realizes a combinatorially expanding two-tile subdivision rule so that the map becomes
expanding. On an intuitive level,  it is very plausible that this should be possible (see the discussion at the beginning of Chapter~\ref{cha:combexp}), but a
rigorous implementation is somewhat cumbersome. 
For this we  define an equivalence relation $\sim$ on $S^2$ that essentially
collapses  components where the map fails to be expanding to
single points. We  show that $\sim$ is of Moore-type and will obtain a suitable Thurston map on the quotient $S^2/\Sim$.




Existence and uniqueness results for \emph{invariant Jordan
  curves} are proved in Chapter~\ref{cha:constructc}. It is 
one of the central chapters  of the present work. Here we
establish Theorems~\ref{thm:main}, \ref{thm:exinvcurvef}, and
\ref{thm:uniqc}, and Corollary~\ref{cor:subdivnlarge} about
existence and uniqueness of invariant curves mentioned in
Section~\ref{sec:int-inv-cell-de}.  One can obtain invariant
curves from an iterative procedure discussed in detail in
Section~\ref{subsec:ittproc}.

In  Section~\ref{sec:cc-quasicircle} we prove that
a Jordan curve $\CC$ is a  \emph{quasicircle} if it is  invariant for an expanding Thurston map
$f$ 
(Theorem~\ref{thm:Cquasicircle}).  If the  cell
decompositions $\DD^n(f,\CC)$, $n\in \N_0$, are obtained from
such an $f$-invariant Jordan curve $\CC$, then the edges in
these cell decompositions are uniform quasiarcs and the
boundaries 
of tiles are uniform quasicircles
(Proposition~\ref{prop:arc}). The underlying metric in all these
statements is any visual metric for $f$.


In Chapter~\ref{cha:combexpfac} we revisit visual metrics. We
introduce the {combinatorial expansion factor} $\Lambda_0(f)$ and
prove Theorem~\ref{thm:visexpfactors1} (see the outline in
Section~\ref{sec:intvismet-vissph}).
In its proof we use the invariant curves constructed in
Chapter~\ref{cha:constructc} to obtain particularly nice
visual metrics for a given expanding Thurston map $f$. They have the property that the map $f$ expands
distances locally by the constant factor $\Lambda$ (see
\eqref{simmetric}).  
   
Chapter~\ref{cha:measure} is devoted to the measure-theoretic
dynamics of an expanding Thurston map $f$.  The main result is
Theorem~\ref{thm:maxentr0} about existence and uniqueness of a
measure of maximal entropy $\nu_f$ for $f$. For the convenience of the reader we review
(mostly standard) material from measure-theoretic dynamics in
Section~\ref{sec:reviewmdyn}. 

The geometry of the visual sphere $(S^2, \varrho)$ of an expanding Thurston map
$f\: S^2\ra S^2$ 
is explored in Chapter~\ref{cha:geom-visu-sphere},
another central part of our work. Here  we show  the important
 fact---already discussed in Section~\ref{sec:intvismet-vissph}---that $f$ is
conjugate to a rational map if and only if $(S^2, \varrho)$ is
quasisymmetrically equivalent to the standard $2$-sphere
 (see Theorem~\ref{thm:S2vsf}~\ref{item:S2qsphere}). In addition, we prove linear
local connectivity  of the visual
sphere (Proposition~\ref{prop:annLLC}), as well as its Ahlfors
regularity in the absence of periodic
critical points of the map (Proposition~\ref{prop:Ahlforsreg}). 


In Chapter~\ref{cha:rati-thurst-maps-1} we study 
measure-theoretic properties of rational expanding Thurston maps
$f\: \CDach\ra \CDach$. We first construct a measure $\lambda_f$
on $\CDach$ that is $f$-invariant and absolutely continuous with
respect to Lebesgue measure on $\CDach$
(Theorem~\ref{thm:ex_inv_abs_L}).  This allows us to apply
methods from ergodic theory.
 As a consequence  we recover (a weak form of)
Zdunik's result that the measure of maximal entropy of $f$ is
absolutely continuous with respect to Lebesgue measure if and
only if $f$ is a Latt\`es map
(Theorem~\ref{thm:abscontimpLattes}).

This in turn allows us to finish the discussion of the visual
sphere of an expanding Thurston map $f$ begun in
Chapter~\ref{cha:geom-visu-sphere}. Specifically, in
Section~\ref{sec:lyapunov-exponent-mu} we prove that the visual sphere of
such a map $f$ is snowflake equivalent to the standard
$2$-sphere if and only if $f$ is topologically conjugate to a
Latt\`{e}s map (Theorem~\ref{thm:S2vsf}~\ref{item:S2Lattes}).

Chapter~\ref{cha:latt-maps-comb} gives
another application of (Zdunik's)
Theorem~\ref{thm:abscontimpLattes}.
By using this theorem we  show that Latt\`es maps can be characterized 
 in terms of their
combinatorial expansion behavior (see Theorem~\ref{thm:Qianthm} 
mentioned in Section~\ref{sec:infures}).




Some further
developments and future perspectives are presented in 
Chapter~\ref{ch:outlook}. Here  we discuss some recent
related work and  open problems.

The appendix is devoted to several  subjects whose inclusion in the main text would have been too distracting. For example, in one of its sections 
we establish a useful variant of Janiszewski's theorem in plane topology, whose proof is a bit technical.  
We also review some fairly 
  standard material about  conformal metrics, Koebe's
distortion theorem, orientation on surfaces, covering
maps, lattices and tori, and  quotient spaces. 
We discuss these topics so that we can refer to them in the main text and 
to make 
this work as self-contained as possible.   

The appendix also contains fairly lengthy sections on  
branched
covering maps, orbifolds, and the canonical orbifold metric. Though the expert will find no surprises here, it is hard to track down this material in an accessible form
in the literature.

\section{List of examples for Thurston maps}
\label{sec:examples}

Throughout the book we consider many examples of Thurston maps
in order to illustrate various phenomena. We list them here with a short
description for the reader's convenience.  The relevant terms used in these
descriptions are defined in later chapters.
Often the maps in our examples are Latt\`{e}s maps. While Latt\`{e}s maps sometimes have special properties  compared  to general Thurston maps, they 
 often provide convenient  examples  with   generic behavior. 
  
\smallskip
A Latt\`{e}s map is discussed  in Section~\ref{sec:Lattes}. In
the terminology of Chapter~\ref{cha:lattes-lattes-type} it is a
flexible Latt\`{e}s map.

In Section~\ref{sec:int-frac-sph} we consider an expanding
Thurston map $h$ that ``generates'' a fractal sphere
$\mathcal{S}$. The map $h$ is not topologically conjugate (or
Thurston equivalent) to a rational map. This is examined in
Example~\ref{ex:obstructedThmap}. Closely related is the fact
that the 
fractal sphere $\mathcal{S}$ is not quasisymmetrically equivalent to the
standard sphere 
(see Example~\ref{ex:non-quasisphere}). 


The examples $f(z)= 1- 2/z^2$ and $g(z) = \frac{\iu}{2}(z+1/z)$
are  used in Section~\ref{sec:defin-thurst-maps} to
explain the concept of  a ramification portrait. Both are in fact
Latt\`{e}s maps.

In Example~\ref{ex:tringflP} our main purpose is to familiarize the reader with the concept of 
Thurston equivalence. To this end we consider two Thurston maps
$f$ and $g$ and show that they are Thurston equivalent. While the map
$g$ is given in a combinatorial fashion and realizes a certain two-tile subdivision rule,
the map $f$ is rational and given by an explicit  formula.

A general construction for Latt\`es-type maps with signature
$(2,2,2,2)$ is presented in Example~\ref{ex:lattes_type}.

In Section~\ref{sec:examples-lattes-maps} we look at several 
Latt\`{e}s maps. 
In Example~\ref{ex:Lattes244} the map is $f(z)=1-2/z^2$, which is a Latt\`{e}s map with
orbifold signature $(2,4,4)$. Similarly, 
in Example~\ref{ex:Lattes333} and Example~\ref{ex:lattes236}
we have Latt\`{e}s maps with orbifold signatures $(3,3,3)$ and $(2,3,6)$,
respectively. 

Certain types of Latt\`es maps are called \emph{flexible}; see
Definition~\ref{def:flex_Lattes} and \eqref{eq:flex_Lattes} for an example.  They have orbifold signature $(2,2,2,2)$,
but not all Latt\`{e}s maps with this signature are
flexible. Such a non-flexible Latt\`es map with signature
$(2,2,2,2)$ is given in Example~\ref{ex:Lattes_Milnor}.

The Thurston map in Example~\ref{ex:no_per_crit_not_exp} has 
 a Levy cycle. 

In Example~\ref{ex:non-expanding-lattes} we present a Thurston
map that is eventually onto, but not expanding. 

In Section~\ref{sec:parab-thurst-polyn} we consider Thurston
maps with signatures $(\infty,\infty)$ and 
$(2,2,\infty)$. They are  equivalent to
$z\mapsto z^n$ (where $n\in \Z\setminus\{-1,0,1\}$) and to
Chebyshev polynomials (up to sign), respectively. 

The example $f(z)= \iu ({z^4-\iu})/({z^4+\iu})$ is used in
Figure~\ref{fig:def_visual_metric} to illustrate the definition
of a visual metric. 

In Example~\ref{ex:f2flaps} we show some cell decompositions 
 generated by an $f$-invariant curve for the map $f(z) =
1-2/z^4$. 

In Example~\ref{ex:subdiv_diff_labels} we consider two maps that
realize two-tile subdivision rules that only differ by the
labeling. 

Several examples of two-tile subdivision rules and the Thurston
maps realizing them are discussed  in
Section~\ref{sec:examples-two-tile}. 
The map in Example~\ref{ex:z2-1} is $f_1(z)=z^2-1$; it realizes
a two-tile subdivision rule that is not combinatorially
expanding.

The  two maps $f_2$ and $\widetilde{f}_2$ in Example~\ref{ex:barycentric} both 
realize the barycentric subdivision rule. The map $f_2$ is a
rational map, but it is not expanding (i.e., its Julia set is not the whole
Riemann sphere $\CDach$). However, the map $\widetilde{f}_2$   is
expanding. It  is an example of an expanding Thurston map with
periodic critical points. 

The map $f_3$ in Example~\ref{ex:obstructed_map} is the map $h$
considered in Section~\ref{sec:int-frac-sph}. It realizes a
certain two-tile subdivision rule and is an obstructed map. This
means that $f_3$ is not Thurston equivalent to a rational map.

The map $f_4$ in Example~\ref{ex:2x3} is a Latt\`{e}s-type map
that is again not Thurston equivalent to a rational map. While it is 
somewhat easier to define than the map $f_3$ in Example
\ref{ex:obstructed_map}, it is less generic, since $f_4$ has a
parabolic orbifold, while $f_3$ has a hyperbolic orbifold. The
map $f_4$ realizes the $2$-by-$3$ subdivision rule. If we equip the underlying
sphere $S^2$ with 
 a suitable visual metric for $f_4$, then  $S^2$ consists 
of two copies of  Rickman's rug.

Example~\ref{ex:R_mario3} provides a whole class of Thurston maps.
One of them  is  $f_5(z)=1-2/z^2$ which
realizes a simple two-tile subdivision rule. By ``adding flaps''
we obtain the other maps. All these maps are rational; in fact,
they are given by an explicit formula, which makes them easy to
understand and visualize.

We discuss a general method for  constructing   Thurston maps
from tilings of the Euclidean or the hyperbolic plane in
Example~\ref{ex:maps_from_tilings}. One can use this to find
Thurston maps 
with arbitrarily large sets 
of postcritical points.  

In Example~\ref{ex:ftilde_not_Thurston} we consider an
equivalence relation of Moore-type on $\CDach$ such that the
map $z\mapsto z^2$ does not descend to a branched covering
map. In contrast, the equivalence relation  on $\CDach$ in Example~\ref{ex:quotient_not_Moore} is not of Moore-type, but  
 $z\mapsto z^2$ descends to a Thurston map on the quotient. 

Example~\ref{ex:extra_vertex_D0} shows
why the additional condition  $\post(f)=\V^0$ in
Theorem~\ref{thm:combexp2} is necessary.  

The Thurston map $f$ in Example~\ref{ex:exp_notcexp}  is not combinatorially expanding, yet Thurston equivalent
to an expanding Thurston map $g$. This shows that the sufficient
condition in Proposition~\ref{prop:combexp} is not necessary.

The map $f$ in Example~\ref{ex:invC} is the same as in
Example~\ref{ex:tringflP}. We use it  to outline  the main ideas of 
Chapter~\ref{cha:constructc}. In particular, we show  how to construct 
an $f$-invariant curve $\widetilde{\CC}$ with
$\post(f)\subset \widetilde{\CC}$ (see
Figure~\ref{fig:invC_constr}).

In Example~\ref{ex:infty_C} we return to  the Latt\`{e}s map $g$
from Section~\ref{sec:Lattes} and prove that it has infinitely many distinct
$g$-invariant curves $\CC$ with $\post(g)\subset \CC$.

In Example~\ref{ex:noinvCC} we exhibit  an expanding Thurston
map $f$ for which no $f$-invariant Jordan curve $\CC$ with
$\post(f)\subset \CC$ exists.

 Remark~\ref{rem:ndep} justifies why   the $f^n$-invariant curve 
$\widetilde{\CC}$ given by Theorem~\ref{thm:main} will in general
depend on $n$. 

In Example~\ref{ex:Cit} we use another Latt\`{e}s map to
illustrate an iterative construction of invariant curves (see
Figure~\ref{fig:Cit}). The invariant curve obtained is quite
``fractal'' (its Hausdorff dimension is $>1$).

Example \ref{ex:Cinv_notcexp} shows what can happen if one of
the necessary conditions in the iterative procedure for
producing invariant curve is violated. Namely, the limiting
object $\widetilde{\CC}$ is not a Jordan curve anymore. The
map here is again a Latt\`{e}s
map.

The map in Example~\ref{ex:rect} (yet another  Latt\`{e}s map) has a 
non-trivial (in particular non-smooth) invariant curve that is
rectifiable.

In Example~\ref{ex:notattained} we exhibit an expanding Thurston
map $f$ that has no visual metric with an expansion factor
$\Lambda$ equal to its combinatorial expansion factor
$\Lambda_0(f) $. Therefore, statement 
\ref{item:visexpfactors2} in Theorem~\ref{thm:visexpfactors1}
cannot be improved in general.

In Example~\ref{ex:ratnotenf} we revisit the two maps from
Example~\ref{ex:barycentric} that realize the barycentric
subdivision rule; these maps show that in
Theorem~\ref{thm:S2vsf}~\ref{item:S2qsphere} we cannot replace
``topologically conjugate to a rational map'' with ``Thurston
equivalent to a rational map''.

\ifthenelse{\boolean{singlechapter}}{

%
%

%

\chapter{Thurston maps} 
\label{cha:thmaps}

In this chapter  we set the stage for our   subsequent developments.
 We will first give a brief review of branched covering 
 maps in Section~\ref{sec:branched-coverings}.
 Then we define
Thurston maps (Section~\ref{sec:defin-thurst-maps}), and what it means for such a map to be expanding (Section~\ref{sec:expansion-1}).   Thurston equivalence is discussed in
Section~\ref{sec:thurston-equivalence}.   As we will see in  Example~\ref{ex:tringflP},
this concept is very useful for clarifying the relation between   maps with   similar dynamical behavior. 

Section~\ref{sec:orbif-assoc-thurst} is devoted to the ramification function and the orbifold associated with  a Thurston map. 
   At the end of this section we will also summarize 
  some facts about the canonical orbifold metric of an orbifold, but we reserved  a more detailed discussion of this topic for  the appendix (see Section~\ref{sec:expratThmaps}).  
  We conclude the chapter with a  discussion of  Thurston's characterization
of rational maps among Thurston maps (see Section~\ref{sec:thurst-class-rati}).

\section{Branched covering maps}
\label{sec:branched-coverings}

 Branched covering maps are  modeled on (non-constant)
 holomorphic maps between Riemann surfaces. As our immediate purpose in this section is to
 prepare the definition of a Thurston map,  we will discuss only some basic facts on this topic 
 and relegate more technical aspects to the appendix (see Section~\ref{sec:appbracovmap} in particular). 
 
We call a $2$-dimensional connected manifold (without boundary) a {\em surface}.  All surfaces that we consider 
 will be  orientable, and we will assume that some fixed orientation has been chosen on such 
 a surface (for a review of orientation and related concepts see Section~\ref{sec:orient}). A surface  is called a {\em topological disk} if it is homeomorphic to the open unit disk $\D=\{z\in \C:|z|<1\}$ in the complex plane. 
   
In the following, $X$ and $Y$ are  two compact and connected oriented  surfaces.  If $f\:X\ra Y$ is  a continuous and surjective map,  then $f$ is
called a 
\defn{branched covering map}\index{branched covering map}\index{map!branched covering}
if
we can write it locally  as the map $z\mapsto z^d$  for some $d\in \N$
in  orientation-preserving homeomorphic  coordinates in
domain and target. More precisely, we require that for each point $p\in
X$ 
there exists $d\in \N$, 
topological disks  $U\subset X$ and $V\sub Y$ 
with $p\in U$ and $q\coloneqq f(p)\in V$, and orientation-preserving homeomorphisms  $\varphi\:U\ra \D$
and $\psi\:V\ra \D$ with $\varphi(p)=0$ and $\psi(q)=0$ such that  
\begin{equation}\label{eq:localpower}
 (\psi \circ f\circ \varphi^{-1})(z)=z^d
 \end{equation} 
for all $z\in \D$. This means that the following diagram commutes: 
\begin{equation*}
  \xymatrix{
    p\in U\subset X \ar[r]^f \ar[d]_\varphi & q\in V\subset
    Y\ar[d]^\psi
    \\
    0\in \D \ar[r]^{z\mapsto z^d} & 0\in \D\rlap{.}
    }
\end{equation*}

For the concept of a branched covering map between non-compact surfaces
see Definition~\ref{def:brcovmap}. If $f\: X\ra Y$ is a branched covering map, then one can find conformal structures on the surfaces $X$ and $Y$ so that $f$ becomes a holomorphic map (see Lemma~\ref{lem:pbackcostr}). In this way, one can
often  derive statements for branched covering maps (as in the ensuing discussion) from analogous statements for holomorphic maps. 

The integer $d\geq 1$ in \eqref{eq:localpower} is uniquely
determined by $f$ and $p$, and called the 
\defn{local degree}\index{local degree}\index{deg@$\deg_f(p), \deg(f,z)$} 
of the
map $f$ at $p$, interchangeably denoted by $\deg_f(p)$ or
$\deg(f,p)$ (depending on whether our emphasis is on the point
$p$ or the map $f$).
A point $c\in X$ with        
$\deg_f(c)\geq 2$ is called a
\defn{critical point}\index{critical!point} of $f$, and  a point in $Y$ that has a critical
point as a preimage a {\em critical value}.\index{critical!value} The
set of all critical points of $f$ is denoted by
$\crit(f)$.\index{crit@$\crit(f)$} Obviously, if $f$ is a branched 
covering map  on $X$, then  $\crit(f)$ is {\em discrete in} $X$, i.e., it has no limit 
points in $X$.  
 Hence $\crit(f)$ is a finite set, because $X$ is assumed to be
compact. Moreover, $f$ is  {\em open} (images of open sets are open),   and  {\em
  finite-to-one} (every 
point in $Y$ has finitely many preimages under $f$). Actually, if
$d=\deg(f)$ is the topological degree of $f$ (see Section~\ref{sec:orient}), then 
\begin{equation}
  \label{eq:sum_deg}
  \sum_{p\in f^{-1}(q)}\deg_f(p)=d
\end{equation}
for every $q\in Y$ (see \cite[Section~2.2]{Ha}). In particular, if $q$
is not a critical value of $f$, then $q$ has precisely $d$
preimages.

We denote by $\chi(X)$ the 
{\em Euler characteristic}\index{Euler characteristic} 
of a compact surface $X$ (see, for example,
\cite[Theorem~2.44]{Ha}). If $X=S^2$ is a $2$-sphere, then
$\chi(S^2)=2$. 
Another case relevant for us is when
$X=T^2$
is a $2$-dimensional torus, in which case $\chi(T^2)=0$.  The
degrees of a branched covering map $f\: X\ra Y$ at critical
points are related to the Euler characteristics of the surfaces
by the 
\emph{Riemann-Hurwitz formula}\index{Riemann-Hurwitz formula} 
(see, for example, \cite[Theorem~A3.4]{HuTeich});
namely,
\begin{equation} 
  \label{eq:Riemann-Hurwitz}
  \chi(X)+\sum_{c\in \crit(f)} (\deg_f(c)-1)= \deg(f)\cdot \chi(Y). 
\end{equation}

Suppose $Z$ is another compact and connected oriented surface, 
and $f\colon X\ra Y$ 
and $g\: Y\ra Z$ are
 branched covering maps. Then 
$g\circ f$ is also a branched covering map 
(see Lemma~\ref{lem:2_3_branched}~\ref{item:2_out3_1}) and
\begin{equation}\label{eq: localdegreemult} 
\deg(g\circ f, x)= \deg(g, f(x))\cdot \deg(f,x)
\end{equation} 
for all $x\in X$ (see Lemma~\ref{lem:locdeg}). In the following, we will often use relation \eqref{eq: localdegreemult}
without mentioning it explicitly.  By counting the number  of preimages of a 
point that is not a critical value of $g\circ f$ we see that 
 $$ \deg(g\circ f)= \deg(g)\cdot \deg(f). $$

\section{Definition of Thurston maps}
\label{sec:defin-thurst-maps}

We will be mostly interested in the case of branched covering maps on 
a $2$-sphere $S^2$.\index{S2@$S^2$}
We use the notation $S^2$ for such a sphere, because we think of
it  as a purely topological object  that should be distinguished from 
the Riemann sphere $\CDach$ or other $2$-spheres with  additional
structure (such as a  Riemann surface structure); in other words,
$S^2$ indicates a topological space homeomorphic to $\CDach$. We
call the topology on $S^2$ the 
{\em given} topology.\index{given topology on $S^2$}
Sometimes it is convenient to express topological notions on $S^2$ in metric terms. Then we choose a {\em base metric} on $S^2$ that induces the given  topology on $S^2$. Such a base metric can be obtained, for example, by pulling back the   chordal metric on $\CDach$ by a homeomorphism 
from $S^2$  onto $\CDach$. 

 Branched covering maps on $S^2$   are  topologically not very different from 
rational maps on the Riemann sphere $\CDach$, 
because 
one can show that 
whenever
$f\: S^2\ra S^2$ is  a branched covering map, 
there exist 
homeo\-morphisms $\psi\: S^2\ra \CDach$ and $\varphi\: S^2\ra \CDach$ such 
that $R= \psi\circ f\circ \varphi^{-1}$ is a rational map on $\CDach$ (see 
Corollary~\ref{cor:brcovratup} for a slightly stronger statement). In other 
words, these maps are not only locally modeled on holomorphic maps as in 
\eqref{eq:localpower}, but even globally. In general though, a 
branched covering map $f\colon S^2\to S^2$  is not topologically conjugate (see Section~\ref{sec:thurston-equivalence})
to a rational map, i.e., the homeomorphisms $\varphi$ and $\psi$ above will be different.  So these maps may exhibit behavior under iteration different from rational maps.

For $n\in \N$ we
 denote   the $n$-th iterate of $f$ by  $f^n$. We set  $f^0\coloneqq \id_{S^2}$, but when we refer to an iterate $f^n$ of $f$, then it is  understood that $n\in \N$.  If $f$ is a branched covering map on
$S^2$, then each iterate  $f^n$ is a also a branched covering map and we have
$\deg(f^n)=\deg(f)^n$.
  
A {\em postcritical point}\index{postcritical point} of $f$ 
 is a point $p\in S^2$ of the form $p=f^n(c)$ with $n\in \N$ and $c\in \crit(f)$. So the set of postcritical points of $f$ is given by\index{post@$\post(f)$}   
  \begin{equation*}  
    \post(f)\coloneqq \bigcup_{n\ge 1} \{f^n(c):c\in \crit(f)\}.
  \end{equation*}
  If the cardinality $\#\post(f)$ is finite,
  then $f$ is called {\em
    postcritically-finite}.\index{postcritically-finite} This is
  equivalent to the requirement that the \emph{orbit}\index{orbit}
  $\{f^n(c) : n\in \N_0\}$ of each critical point $c$ of $f$ is
  finite.

 For $n\in \N$ we have 
  \begin{equation}\label{eq:critpfn}
   \crit(f^n)=\crit(f)\cup f^{-1}(\crit(f))\cup \dots \cup f^{-(n-1)}(\crit (f)).
   \end{equation} 
This implies that 
 $\post (f^n)=\post (f)$ and 
  \begin{equation}\label{eq:critpfn2}
  \post(f)=\bigcup_{n\in \N} f^n( \crit(f^n)). 
   \end{equation} 
   
 So  $\post(f)$ is equal to the union of the critical values of
 all iterates $f^n$, for $n\in \N$. 
  It follows that  $f^n \: S^2 \setminus f^{-n}(\post(f)) \to S^2 \setminus
\post(f)$  is a covering map  (see Lemma~\ref{lem:brcovcov}; we give a review of covering maps in    Section~\ref{sec:covmaps}). This implies that away from $\post(f)$ all ``branches  of the inverse of $f^n$'' are defined; more precisely, if $U\sub S^2\setminus \post(f)$ is a path-connected and  simply connected  set, $q\in U$, and $p\in S^2$ a point with $f^n(p)=q$, then there exists 
 a unique continuous map $g\: U\ra S^2$ with 
 $g(q)=p$ and $f^n\circ g=\id_U$ (this easily follows from Lemma~\ref{lem:liftsofcov}).   We refer to such a right inverse of $f^n$ informally as a ``branch of $f^{-n}$''.

We can now record the definition of  the main object of investigation in this work.

\begin{definition}[Thurston maps]\label{def:f}
  A \defn{Thurston map}\index{Thurston map} is a branched covering map
  $f\: S^2\ra S^2$ on a $2$-sphere $S^2$ with $\deg(f)\ge 2$ and a
  finite set of postcritical points.   
\end{definition}

Note that for a branched covering map $f\: S^2\ra S^2$ the condition $\deg(f)\ge 2$ is equivalent to the requirement that $f$ is not a homeomorphism.

Away from its finitely many critical points, a Thurston map is an 
orientation-preserving local homeomorphism on the oriented sphere $S^2$
(see the end of Section~\ref{sec:orient} for the relevant terminology here). 
Since $\post(f^n)=\post(f)$, each iterate $f^n$, $n\in \N$, of a Thurston map is also a Thurston map. 

If $S^2=\CDach$ is the Riemann sphere and $f\: \CDach\ra \CDach$ is a non-constant holomorphic map, then $f$ is a rational map  (and can be represented as the quotient of two polynomials).  If, in addition, $f$ is a Thurston map (i.e., it is postcritically-finite and 
satisfies $\deg(f)\ge 2$), then we call $f$ 
 a {\em rational Thurston map}.\index{Thurston map!rational}  Note that if $f=P/Q$, where $P,Q\ne0$  are polynomials without common zero, then
  $\deg(f)=\max\{\deg(P), \deg(Q)\}$. Here  the degree $\deg(P)$ of a polynomial $P\ne 0$ 
  is equal to $n\in \N_0$ if $ P(z)=a_nz^n+\dots +a_0$ with  $a_0,\dots, a_n\in \C$ and 
  $a_n\ne 0$. This agrees with the topological degree of $P$ considered as a  map  $P\: \CDach\ra \CDach$.

There are no Thurston maps with $\#\post(f)\in \{0,1\}$ (see Corollary~\ref{cor:post012}), and all Thurston maps with $\#\post(f)=2$ are
Thurston equivalent (see Section~\ref{sec:thurston-equivalence}) to a
map $z\mapsto z^k$, $k\in \Z\setminus\{-1,0,1\}$, on the Riemann
sphere (see Proposition~\ref{prop:post2}). Postcritically-finite
rational maps give a large class of Thurston
maps.  Examples include $P(z)= z^2 -1$ (commonly known as the
\emph{Basilica map}, since its Julia set supposedly resembles St
Mark's Basilica reflected in  water), $f(z)= 1- 2/z^2$, or $g(z)=
\iu/2(z+ 1/z)$. Many other examples can be found throughout this work (see also \cite{BBLPP}).

 Since the orbit of each critical point of a Thurston map $f$ is
 finite, it is often convenient to represent such  orbits  by the
 \emph{ramification portrait}\index{ramification!portrait}. 
This is a directed graph,
 where the vertex set $V$ is the union of the orbits of all critical
 points. For $p,q\in V$ there is a directed  edge from $p$ to $q$ if and only
 if $f(p)=q$. Moreover, if  $p$  is a critical point with
 $\deg_f(p)=d$ we label this  edge by ``$d:1$''. 

For example, the map $f(z)= 1-2/z^2$ has the  ramification portrait
\begin{equation*}
  \xymatrix @R=1pt{
    0 \ar[r]^{2:1} & \infty  \ar[r]^{2:1} & 1 \ar[r] & -1 \ar@(r,u)[]
    }
\end{equation*}
and for   $g(z) = \frac{\iu}{2}(z+ 1/z)$ we obtain: 
\begin{equation*}
  \xymatrix @R=1pt{
    1 \ar[r]^{2 : 1} & \iu \ar[dr]  & &
    \\
    &  & 0\ar[r] & \infty\rlap{.} \ar@(r,u)[]
    \\
    -1 \ar[r]^{2 : 1} & -\iu \ar[ur] & & 
  }
\end{equation*}

\section{Definition of expansion}
\label{sec:expansion-1}

Let $f\:S^2\ra S^2$ be a Thurston map and  $\CC$ be a {\em Jordan  curve} in $S^2$ (i.e., a set homeomorphic to the unit  circle in $\R^2$) 
 with $\post(f)\sub \CC$.
We  fix a base metric $d$ on $S^2$ that induces the given  topology on $S^2$. 
For $n \in \N$ we denote by $\mesh(f,n,\CC)$\index{mesh@$\mesh$} the
supremum  of the  diameters of all connected components of the set
$f^{-n}(S^2\setminus \CC)=S^2\setminus f^{-n}(\CC)$.

\begin{definition}[Expansion]
  \label{def:exp}
  A Thurston map $f\: S^2\ra S^2$ is called
  \defn{expanding}\index{expanding|textbf}\index{Thurston map!expanding|textbf} 
  if there exists a Jordan curve $\CC$ in $S^2$
  with $\post(f)\sub \CC$ and  
  \begin{equation} \label{def:expanding}
    \lim_{n\to \infty}   \mesh(f,n,\CC) =0. 
  \end{equation}
\end{definition} 

We will study the concept of expansion in more detail in Chapter~\ref{cha:expansion} after we have built up some methods for a  systematic investigation. For the moment we summarize some main facts related to  this concept. 

The  set $f^{-n}(S^2\setminus \CC)$ has actually only
  finitely many components;  so the supremum in the definition of
  $\mesh(f,n,\CC)$ is a maximum (see
  Proposition~\ref{prop:celldecomp}~\ref{item:nedgesC}).
  We will see in Lemma~\ref{lem:exp_ind_C}  that if condition 
  \eqref{def:expanding}  is satisfied for one Jordan curve
  $\CC\subset S^2$ with $\post(f) \subset \CC$,
then it actually holds  for every such curve. So 
  expansion is a property of the map $f$ alone.
 Moreover,  it  is really a topological
  property, since  it is independent of the choice of the base metric $d$ 
  on $S^2$ (as long as $d$  induces the given  topology of $S^2$).
  Our notion  of expansion for a  Thurston map is equivalent to 
  a similar concept of expansion  introduced by Ha\"\i ssinky-Pilgrim (see 
  \cite[Section~2.2] {HP} and Proposition~\ref{prop:expequivexp}).

  It is immediate that expansion is preserved under topological conjugacy, i.e., if
  $f$ and $g$ are topologically conjugate Thurston maps, then $f$ is
  expanding if and only if $g$ is expanding.  On the other hand, expansion is not preserved under Thurston
  equivalence (see the next section for the terminology, and 
  Example~\ref{ex:barycentric}).  
 A related fact is that if   two
  expanding Thurston maps are Thurston equivalent, then they  are  actually   topologically conjugate (see Theorem~\ref{thm:exppromequiv}).  
  
   Expansion is compatible with iteration of the map. Namely, if $f\colon
    S^2\to S^2$ is a Thurston map, and $F=f^n$, $n\in \N$,  is an iterate, then $f$
    is expanding if and only if $F$ is expanding
    (Lemma~\ref{lem:Thiterates}). 
 
A map $f\colon S^2\to S^2$ is called 
\emph{eventually onto},\index{eventually onto} 
if for any non-empty open set $U\subset S^2$ there is an iterate $f^n$
  such that $f^n(U)=S^2$. Every expanding Thurston map is eventually
  onto (Lemma~\ref{lem:event_onto}). The   converse does not hold:  there are Thurston maps
  that are eventually onto, but not expanding (see
  Example~\ref{ex:non-expanding-lattes}).
  
   If $f\: S^2\ra S^2$ is a branched covering map, then a point $p\in
   S^2$ is called 
{\em periodic}\index{periodic!point} 
   if  
 $f^n(p)=p$ for some $n\in \N$. The smallest $n$ for which this is true
is called the {\em period}   of the periodic point. 
  The point $p$ is called  
 {\em preperiodic}\index{preperiodic!point} if there exists $k\in \N_0$ such that $q=f^k(p)$ is periodic.  
 Finally, a    
 \emph{periodic critical point}\index{periodic!critical point}\index{critical!point!periodic} 
 is
  a periodic point $c\in \crit(f)$.  
  
 The following statement gives 
 a criterion when a rational Thurston map is
expanding.

\begin{prop} \label{prop:rationalexpch}
Let $R\:\CDach\ra \CDach$ be a rational 
Thurston map. 
\index{Thurston map!rational} 
Then the following conditions are equivalent:

\begin{enumerate} 

\item $R$ is expanding.
  \label{item:rationalexpch1}
  \index{expanding} 
  \index{Thurston map!expanding}
\item
  \label{item:rationalexpch2}
  The Julia set of $R$ is equal to $\CDach$. 

\item
  \label{item:rationalexpch3}
  $R$ has no periodic critical points. 
\end{enumerate}

\end{prop}

An immediate consequence of this proposition is that no
postcriti\-cal\-ly-finite polynomial $P\: \CDach \ra \CDach$
(with $\deg(P)\ge 2)$ can be expanding. Indeed, in this case
$\infty\in \CDach$ is both a critical point and a fixed point of
$P$. So condition \ref{item:rationalexpch3} is violated. For a
related fact see Lemma~\ref{lem:poly_not_exp}.

In general,  expanding Thurston maps may have 
periodic critical points (see Example~\ref{ex:barycentric}). On the
other hand, there are Thurston maps that do not have 
periodic critical points, yet are not Thurston equivalent to any expanding
map (see Example~\ref{ex:no_per_crit_not_exp}).

Our proof of Proposition~ \ref{prop:rationalexpch}    relies on  facts 
(in particular, on Lem\-ma~\ref{lem:event_onto},  Lemma~\ref{lem:length_exp}, and 
Proposition~\ref{lem:R_orbimetric}) that we will establish    later.  
We will  also  use some basic concepts from complex
dynamics, which can be found in \cite{CG} and \cite{Mi}, for example.   

\begin{proof}[Proof of Proposition~\ref{prop:rationalexpch}] 
  We will show the 
  chain of implications
  \ref{item:rationalexpch1} $\Rightarrow$ \ref{item:rationalexpch2}
$\Rightarrow$ \ref{item:rationalexpch3}
$\Rightarrow$ \ref{item:rationalexpch1}. 


\smallskip
\ref{item:rationalexpch1}
$\Rightarrow$
\ref{item:rationalexpch2}
Let $R\colon \CDach\to \CDach$ be a rational expanding Thurston map. Then its Julia set 
$\mathcal{J}\sub \CDach$ is non-empty. Let $\mathcal{F}= \CDach
\setminus \mathcal{J}$ be the Fatou set of $R$. By 
Lemma~\ref{lem:event_onto} the map  $R$ is eventually onto. So if we assume 
$\mathcal{F}\neq \emptyset$, then, since $\mathcal{F}$ is open,  this implies that $R^n(\mathcal{F})=
\CDach$ for a sufficiently high  iterate $R^n$.   
Now  $\mathcal{F}$ is invariant for $R$; so  this means that $\mathcal{F}=\CDach$ and 
$\mathcal{J}=\emptyset$. This is a contradiction.

\medskip
\ref{item:rationalexpch2}
$\Rightarrow$ 
\ref{item:rationalexpch3}
If the Julia set of $R$ is equal to $\CDach$, then its Fatou set
is empty. This implies that $R$ cannot have periodic critical
points, because a periodic critical point of a rational map is
part of a super-attracting cycle and belongs to the Fatou set.

\medskip
\ref{item:rationalexpch3}
$\Rightarrow$
\ref{item:rationalexpch1}
Suppose $R$ has no periodic critical points. Then there exists a
geodesic metric $\om$ on $\CDach$ (the canonical orbifold metric
of $f$; see Section~\ref{sec:orbif-assoc-thurst} and Section~\ref{sec:expratThmaps}) such that $R$ expands
the $\om$-length of each path in $\CDach$ by a fixed factor
$\rho>1$ (Proposition~\ref{lem:R_orbimetric}). This implies that
$R$ is expanding (Lemma~\ref{lem:length_exp}).
\end{proof}

 


\section{Thurston equivalence}
\label{sec:thurston-equivalence}

Suppose  $f\colon S^2\to S^2$ and $g\colon
\widehat{S}^2 \to \widehat{S}^2$ are two Thurston maps. Here $\widehat S^2$ is another
$2$-sphere. Often $S^2=\widehat S^2$, but sometimes  it is important
to distinguish the spheres on which the Thurston maps are defined.  
We call the maps $f$ and $g$ \defn{topologically
  conjugate}\index{topologically conjugate|textbf}
\index{conjugacy} 
 if there exists a
homeomorphism $h\: S^2\ra \widehat S^2 $ such that $h\circ f = g\circ
h$. This defines a notion of equivalence for  Thurston maps. Topologically conjugate maps have essentially the same dynamics under iteration up to ``change of coordinates''. 

It is often convenient to consider a weaker notion of equivalence for
Thurston maps. To define it, we first recall the definition of   homotopies and isotopies between spaces. 
Let $I=[0,1]$, and $X$, $Y$  be  topological  spaces. A {\em
  homotopy  between $X$ and $Y$}\index{homotopy}
 is a continuous map $H\: X\times I \ra Y$. We define  $H_t\coloneqq H(\cdot, t)\: X\ra Y$ for $t\in I$. The map $H_t$  is  called the {\em time-$t$ map} of the homotopy.  A homotopy 
$H\:  X\times I \ra Y$ is called an 
{\em isotopy}\index{isotopy}  
if 
$H_t$ is a homeomorphism of $X$ onto $Y$ for each $t\in I$. If $X=Y$, then 
$H$ is called a homotopy (or isotopy) {\em on} $X$. 

Let $A\sub X$. If  $H\: X\times I \ra Y$ is   a homotopy,
then 
we say that  $H$ is a homotopy  
{\em relative}\index{homotopy!rel.\ $A$}
to $A$ (abbreviated ``$H$ is a homotopy  rel.\ $A$'') if 
$H_t(a)=H_0(a)$ for all   $a\in A$ and  $t\in I$. So this means
that the image of each point in $A$ remains fixed during the homotopy.
Similarly, we speak of 
{\em isotopies rel.\ $A$}.\index{isotopy!rel.\ $A$}  
Two homeomorphisms $h_0, h_1\: X\ra Y$ are called  {\em  isotopic
  rel.~$A$} if there exists an isotopy $H\: X\times I \ra Y$
rel.~$A$ with $H_0=h_0$ and  $H_1=h_1$.

 Let $B$ and $C$ be subsets of $X$. We say that
$B$ is 
{\em isotopic to $C$ rel.\ $A$},\index{isotopy!rel.\ $A$} 
or {\em $B$ can be isotoped (or  deformed)  into $C$ rel.\ $A$}, if there
exists an isotopy $H\: X\times I\ra X$ rel.\ $A$ with $H_0=\id_X$
and $H_1(B)=C$.  Note that this notion depends on the ambient
space $X$ containing the sets $A$, $B$, $C$.

\begin{definition}[Thurston equivalence]\label{def:Thequiv}
Let $f\: S^2\ra S^2$ and $g\:\widehat S^2\ra \widehat S^2$
be Thurston maps. 
Then they are called \defn{(Thurs\-ton)  
  equivalent}\index{Thurston!equivalent|textbf} if there exist homeo\-morphisms 
  $h_0,h_1\:S^2\ra \widehat S^2 $ that are  isotopic rel.\ $\post(f)$  and satisfy
 $    h_0\circ f = g\circ h_1$.
 \end{definition} 
In this case, the following  diagram commutes: 
\begin{equation}\label{Thequiv1} 
  \xymatrix{
    S^2 \ar[r]^{h_1} \ar[d]_f & \widehat S^2 \ar[d]^g
    \\
    S^2 \ar[r]^{h_0} & \widehat  S^2\rlap{.}
  }
\end{equation}

Our $2$-spheres $S^2$ and $\widehat S^2$ are assumed to be
oriented.  Since the homeomorphisms $h_0$ and $h_1$ are isotopic,
they are either both orien\-tation-preserving or
orientation-reversing.  One defines Thur\-ston
equi\-valence  sometimes differently by insisting on the homeomorphisms $h_0$
and $h_1$ in Definition~\ref{def:Thequiv} being
orientation-preserving. In this case, we say that $f$ and $g$ are
{\em orientation-preserving Thurston equivalent}.
\index{Thurston!equivalent!orientation-preserving}
\index{orientation!preserving!Thurston equivalent}

This is a stronger notion of Thurston equivalence.  For example,
according to our definition, a rational Thurston map $R$ on
$\CDach$ is Thurston equivalent to the map
$z\mapsto \overline{R(\overline{z})}$ (we use
$h_0(z)=h_1(z)=\overline z$ in Definition~\ref{def:Thequiv}), but
these maps are in general not orientation-preserving Thurston
equivalent.
 

  If two Thurston maps are topologically conjugate, then they are
  Thurston equivalent. However, 
 they will not be
  orientation-preserving Thurston equivalent  in general.
 This is the main reason why we use  our more general concept of Thurston equivalence.

  \begin{lemma}
    \label{lem:T-eq_crit_post}
    Let $f\:S^2\ra S^2$ be a Thurston map, $g\: \widehat S^2\ra \widehat S^2 $ be  a continuous map, and $h_0,h_1\: S^2\ra  \widehat S^2$ be homeomorphisms 
    that are isotopic rel.\ $\post(f)$ and satisfy
   $ h_0\circ f=g\circ h_1$. Then $g$ is also a
     Thurston map  and we have 
    \begin{align*}
      \crit(g)&=h_1(\crit(f)), 
      \\
      \notag
      \post(g)&=h_0(\post(f))=h_1(\post(f)).
  \end{align*}
  \end{lemma}
   Note that we again have  the diagram~\eqref{Thequiv1}. So 
   under the given assumptions,   $f$ and $g$ are Thurston equivalent. 
  If we already know that $g$ is a Thurston map (which was part of the conclusion of the lemma), then,  as the proof will show,  the statements  about $\crit(g)$ and $\post(g)$ are valid under the weaker assumptions that $h_0$ and $h_1$ are homeomorphisms satisfying $h_0\circ f=g\circ h_1$ and  
$h_0|\post(f)=h_1|\post(f)$, but are not necessarily isotopic rel.\ $\post(f)$. 
 
  \begin{proof}  Since $h_0$ and $h_1$ are isotopic, they both preserve or both reverse orientation.  
So it is clear that $g=h_0\circ f\circ h_1^{-1}$  is a branched covering map with   $\crit(g)=h_1(\crit(f))$. 
    Thus 
    $$g(\crit(g))=
    (g\circ h_1)(\crit(f)) = (h_0 \circ f)(\crit(f)).$$ 
  Since  $h_0|\post(f)=h_1|\post(f)$ and $f^n(\crit(f))\sub\post (f)$
    for all $n\in \N$,  we  inductively derive 
    $$g^n(\crit(g))=h_0(f^n(\crit(f)))=h_1(f^n(\crit(f)))$$ for all $n\in \N$. Hence 
    \begin{align*}
      \post(g)&=\bigcup_{n\in \N}g^n(\crit(g))=\bigcup_{n\in \N}
      h_0(f^n(\crit(f)))
      \\
      &=h_0(\post(f))=h_1(\post(f)).
    \end{align*}
 In particular,    this shows that $\post(g)$ is finite. Since 
 $\deg(g)=\deg(f)\ge 2$, it follows that  $g$ is a Thurston map. 
\end{proof}  

The previous lemma implies that the roles  of the maps $f$ and $g$ in Definition~\ref{def:Thequiv} are symmetric. Indeed, 
suppose  $H\colon S^2\times I\ra \widehat{S}^2$ is an  isotopy rel.\ $\post(f)$
with $H_0=h_0$ and $H_1=h_1$. Define  $K\colon \widehat{S}^2 \times I \ra
S^2$ as  $K(\hat x, t) = (H_t)^{-1}(\hat{x})$ for $\hat{x}\in \widehat{S}^2$, $t\in I$. 
The map $(x,t)\mapsto (H(x,t),t)$ is a continuous bijection between  the compact Hausdorff spaces $S^2\times I$ and $\widehat{S}^2\times I$, and hence a homeomorphism. Its inverse is given by $(\hat x, t)\mapsto (K(\hat x, t),t)$. This implies that $K$ is continuous, and so  obviously  an isotopy rel.\ $h_0(\post(f))=h_1(\post(f))=\post(g)$.
Since $k_0\coloneqq h_0^{-1}=K_0$ and $k_1\coloneqq h_1^{-1}=K_1$, the homeomorphisms $k_0$ and $k_1$ are isotopic rel.\ $\post(g)$.  We also have 
$$k_0\circ g= h_0^{-1}\circ g=  f\circ h_1^{-1}=  f \circ k_1,  $$ and so the conditions in Definition~\ref{def:Thequiv} are also satisfied if we interchange  $f$ and $g$.

It is also clear  that if $f,g,h$
are Thurston maps, $f$ is (Thurston) equivalent to $g$, and $g$ is
equivalent to $h$, then $f$ is equivalent to $h$. Thus,  Thurston
equivalence leads to a notion of  equivalence for Thurston maps. 

If two Thurston maps $f$ and $g$ are equivalent,  then for each 
$n\in \N$ the iterates  $f^n$ and $g^n$ are also equivalent. To show this
statement, one needs to lift the isotopy rel.\ $\post(f)$ between
the homeomorphisms $h_0$ and $h_1$ as in 
Definition~\ref{def:Thequiv}. We postpone the proof to Chapter~\ref{cha:iso}, where such lifts are considered (see Corollary~\ref{cor:fg_eq_fngn_eq}).

  \begin{ex} \label{ex:tringflP}  Up to  Thurston equivalence,  one can often 
    construct combinatorial models of maps that are given in some other specific way, for example by an analytic formula. To illustrate this, we consider the map 
     $f\colon
  \CDach\to\CDach$ given by
\begin{equation} \label{ex:tringflapsform} 
 f(z)= 1+ (\omega- 1)/z^3,
 \end{equation} 
  where $\omega= e^{4\pi \iu/3}$.  
 Then  $\crit(f)=\{0,\infty\}$, and  $f$ 
    has the ramification
portrait
\begin{equation}
  \label{eq:crit_port_tria_1flap}
  \xymatrix @R=1pt{
    0 \ar[r]^{3:1} & \infty \ar[r]^{3:1} & 1\ar[r] &  \omega\rlap{.} \ar@(r,u)[] 
  }
\end{equation}
So $\post(f)=\{1, \omega, \infty\}$, and $f$ is a Thurston map.   

\ifthenelse{\boolean{nofigures}}{}{  
  \begin{figure}
    \centering
    \begin{overpic}
      [width=12cm, 
      tics=20]{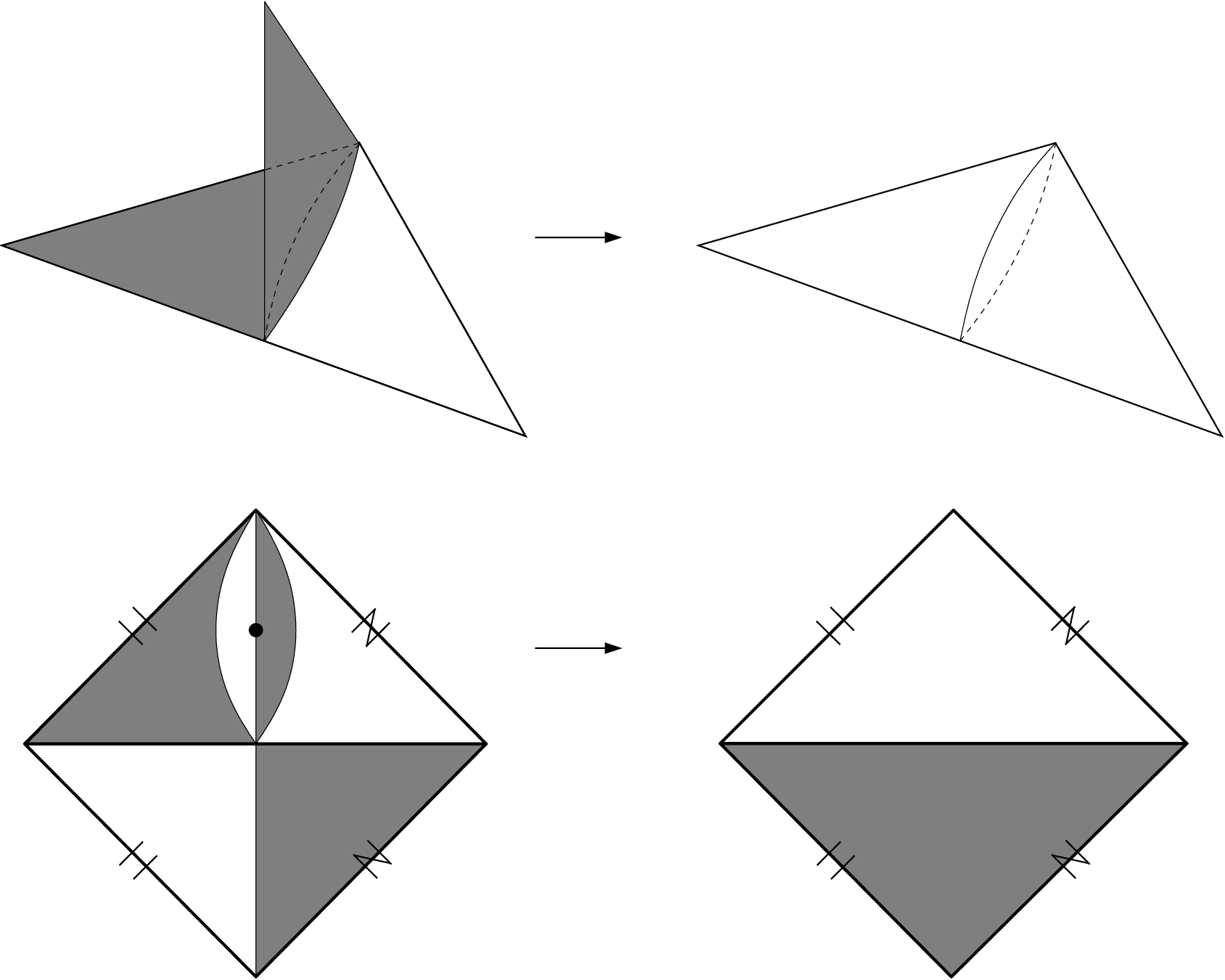}
      \put(97,60){$\mathcal{S}^0$}
      \put(38,60){$\mathcal{S}^1$}
      \put(56,57.5){$\scriptstyle{\omega}$}
      \put(100,42){$\scriptstyle{1}$}
      \put(85,69.5){$\scriptstyle{\infty}$}
      \put(41,42){$\scriptstyle{1\mapsto\omega}$}
      \put(30,69.5){$\scriptstyle{\infty\mapsto 1}$}
      \put(23,79.5){$\scriptstyle{\overline{\omega}\mapsto \omega}$}
      \put(20,48.5){$\scriptstyle{0\mapsto \infty}$}
      \put(-3,56){$\scriptstyle{\omega\mapsto \omega}$}
      \put(46,29){${g}$}
      \put(56,20.5){$\scriptstyle{\omega}$}
      \put(97.5,20){$\scriptstyle{1}$}
      \put(76.5,39){$\scriptstyle{\infty}$}
      \put(76.5,-2){$\scriptstyle{\infty}$}
      \put(40,20){$\scriptstyle{1\mapsto\omega}$}
      \put(20,-2){$\scriptstyle{\infty\mapsto 1}$}
      \put(20,39){$\scriptstyle{\infty\mapsto 1}$}
      \put(18.3,29.8){$\scriptstyle{\overline{\omega}\mapsto \omega}$}
      \put(-2,22){$\scriptstyle{{\omega}\mapsto \omega}$}
      \put(22.5,20){$\scriptstyle{0\mapsto \infty}$}
      %
      \put(27,24){$\scriptstyle{Y_1}$}
      \put(21,24.7){$\scriptstyle{Y_2}$}
      \put(18.2,24.7){$\scriptstyle{Y_3}$}
      \put(12,24){$\scriptstyle{Y_4}$}
      \put(14,13){$\scriptstyle{Y_5}$}
      \put(25,13){$\scriptstyle{Y_6}$}
      %
      %
      \put(30,55){$\scriptstyle{Y_1}$}
      \put(23.5,69){$\scriptstyle{Y_2}$}
      \put(15,59){$\scriptstyle{Y_4}$}
      %
    \end{overpic}
    \caption{The map $g$.}
    \label{fig:triangle_flap0}
  \end{figure}
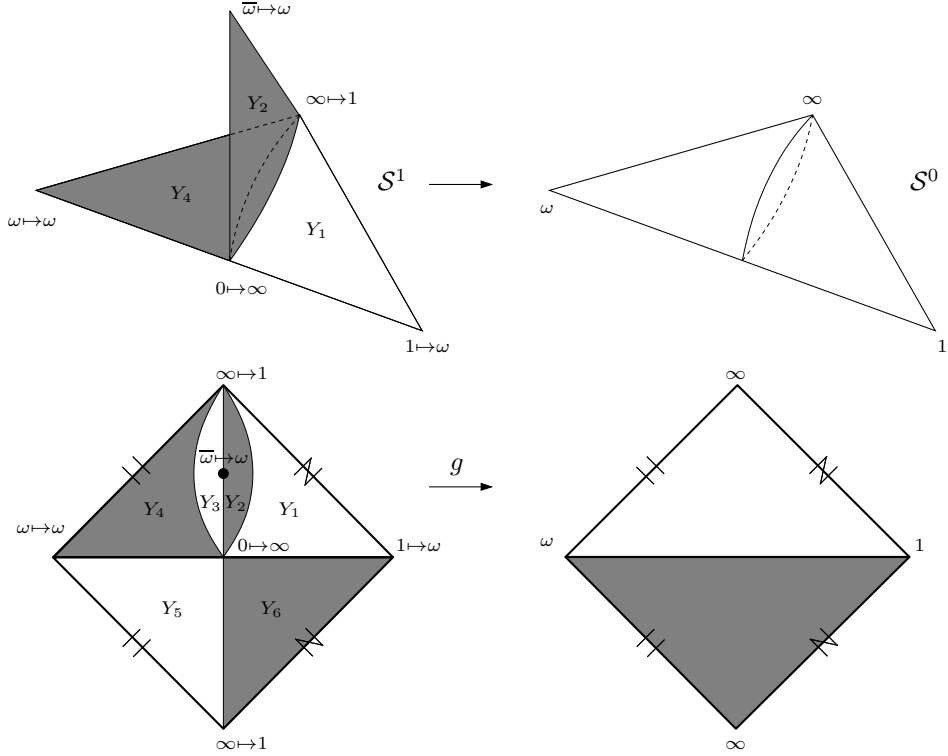}

  We   now construct  a related   Thurston map $g\colon S^2\to S^2$  in a similar fashion as the map $h$ in
    Section~\ref{sec:int-frac-sph}.  For this 
   let $T$ be a right-angled   isosceles 
    Euclidean triangle whose hypotenuse has length $1$. The  angles of $T$ are then $\pi/2, \pi/4,
    \pi/4$. We also consider a triangle $T'$ that is similar to $T$ by the scaling factor $\sqrt{2}$. 
    We glue two copies of  $T$ together along their boundaries to form a
    pillow $\mathcal{S}^0$ (see Section~\ref{sec:expratThmaps}),  which  is a
    topological $2$-sphere. These  two triangles  are called the $0$-tiles.
    As before, we color one of them  white,  and the other black. 



    We divide each of the two $0$-tiles by the perpendicular
    bisector of the hypotenuse into two triangles similar to $T$
    and isometric to $T'$.  We slit the pillow open along one
    such bisector and glue two copies of $T'$ into the slit as
    indicated on the left in  Figure~\ref{fig:triangle_flap0}.
    This results in a polyhedral surface $\mathcal{S}^1$
    consisting of six triangles, each isometric to $T'$.  These
    six small triangles are called the $1$-tiles. They are
    colored in a checkerboard fashion black and white so that
    triangles sharing an edge have different colors, as indicated
    in the picture. Note that there are two points (labeled
    ``$0\mapsto \infty$'' and ``$\infty\mapsto 1$'') in which all
    $1$-tiles intersect.

    Each small triangle in $\mathcal{S}^1$ is now mapped by a
    similarity to the triangle in $\mathcal{S}^0$ of the same
    color. This  determines a unique map  $\mathcal{S}^1\to \mathcal{S}^0$ if the vertices of the $1$-tiles are mapped as indicated  at the top in  Figure~\ref{fig:triangle_flap0}.  
    
    We identify $\mathcal{S}^1$ with $\mathcal{S}^0=S^2$ such that
    the four triangles shown on the top left in
    Figure~\ref{fig:triangle_flap0} are identified with the white
    triangle in $\mathcal{S}^0$, and the other two small triangles in
    $\mathcal{S}^1$ are identified with the black triangle in
    $\mathcal{S}^0$. Here we   require that this identification  
 respects   the natural identification  of     the three common corners of  the two copies of $T$ (labeled $\om$, $1$, $\infty$ in the picture) which are  contained  in both $\mathcal{S}^0$ and $\mathcal{S}^1$.   
    This yields a Thurston map $g\colon S^2\to S^2$, indicated at  the
    bottom in Figure~\ref{fig:triangle_flap0}. Here we have cut the
    pillow $\mathcal{S}^0$ along the two legs of the two
    triangles. The two pairs of legs marked with the same symbol have
    to be identified, i.e., glued together, to form the pillow. 

    Note that $g$ depends on how precisely $\mathcal{S}^1$ is
    identified with $\mathcal{S}^0$. A different identification
    yields another Thurston map $\widetilde{g}$, but it is easy
    to see that $\widetilde{g}$ is always Thurston equivalent to
    $g$. Thus the notion of Thurston equivalence allows us some
    latitude in specifying the precise identification of
    $\mathcal{S}^1$ with $\mathcal{S}^0=S^2$.
    
    Our main point here is the following statement:
    {\em  
  The map $f$ as defined in \eqref{ex:tringflapsform} and the map $g$ constructed above are Thurston
  equivalent.}  As our purpose  is to give  the reader some intuition for the concept of Thurston equivalence, we will only outline a proof omitting some details.

Note that $f(z)= \tau(z^3)$, where $\tau(\zeta)= 1+ (\omega-1)/\zeta$
is a M\"{o}bius transformation that maps the upper half-plane to the
half-plane above the line through the points $\omega$ and $1$ (indeed,
$\tau$ maps $0,1,\infty$ to $\infty, \omega,1$, respectively).

    Let $\CC\subset \CDach$ be the circle through $\omega,1,\infty$ (i.e.,
    the extended line through $1$ and $\omega$). Then $\CC$ contains all
    postcritical points of $f$. The closures of the two components of
    $\CDach \setminus \CC$ are called the $0$-tiles (of $f$).
    The $0$-tile containing $0\in \CDach$ (i.e., the half-plane above
    the line through $1$ and $\omega$) is colored white, the other $0$-tile
    is colored black.

  Since $f(z)=\tau(z^3)$,   we have 
    $f^{-1}(\CC)= \bigcup_{k\in \{1,\dots, 6\}} R_k$, where 
    $$R_k=\{re^{\iu k\pi/3} : 0\leq r\leq \infty\}$$ for $k\in \Z$
    is the ray from $0$ through the sixth root of unity $e^{\iu
      k\pi /3}$ (note that $R_{0}= R_{6}$). These rays divide $\CDach$
    into open sectors that are the complementary  
      components of $\CDach \setminus f^{-1}(\CC)$. For $k=1, \dots, 6$ let
      \begin{equation*}
  X_k \coloneqq  \{r e^{\iu t} : 0\leq r \le  \infty,\,  (k-1)\pi/3 \leq t \le 
  k\pi/3\}.
\end{equation*} be the closure of the sector bounded by $R_{k-1}$ and $R_k$. 
    We call these sets the $1$-tiles. Note that $f$ maps each $1$-tile $X_k$ 
        homeomorphically to a $0$-tile $X^0\subset\CDach$. We color $X_k$
    white if $X^0$ is white, and black otherwise. Then the $1$-tile $X_k$, $k=1, \dots , 6$,  is colored black or white depending on whether $k$ is even or odd. 

   In order to show that $f$ and $g$ are Thurston equivalent, we need two homeomorphisms $h_0$ and $h_1$ as in Definition~\ref{def:Thequiv}. The  homeomorphism $h_0 \colon \CDach \to S^2=\mathcal{S}^0 $ is defined  as follows. We map 
    the  white and  black $0$-tiles in $\CDach$ homeomorphically 
    to the white and black $0$-tile in $\mathcal{S}^0$, respectively. We can do this so that  the points
    $\omega, 1, \infty\in \CDach$ are mapped to the points labeled
    $\omega,1, \infty\in S^2$ on the top  right  in 
    Figure~\ref{fig:triangle_flap0}, and so that the homeomorphisms on these two $0$-tiles match along the common boundary.   

The homeomorphism $h_1\colon \CDach \to \mathcal{S}^1$ is constructed similarly by mapping $1$-tiles in $\CDach$ homeomorphically to corresponding $1$-tiles in $\mathcal{S}^1$. We  will then view  $h_1$ as a map to the domain  $S^2=\mathcal{S}^0$ of $g$ by using   the given identification of $\mathcal{S}^1$ with  
$\mathcal{S}^0$  

 For the precise definition of $h_1$, we denote  the $1$-tiles 
 in $\mathcal{S}^1$ by  $Y_1, \dots , Y_6$  so that they follow in positive cyclic order 
around the point labeled ``$0\mapsto \infty$'' on the  left in
Figure~\ref{fig:triangle_flap0} and such that $Y_1$  is the  white
$1$-tile containing the point labeled ``$1\mapsto
\omega$''.  
With our labeling the  $1$-tile $X_k$   
     in $\CDach$ has the same color as the $1$-tile $Y_k$   in $\mathcal{S}^1$.
     So for fixed  $k=1, \dots, 6$, the map  
   $f|X_k\colon X_k\to X^0$ is a homeomorphism of $X_k$ onto a $0$-tile $X^0$ in $\CDach$, and 
    $g|Y_k \colon Y_{k}\to Y^0$ is a  homeomorphism of $Y_k$ onto a $0$-tile 
    $Y^0$ in $\mathcal{S}^0$, where  the tiles  $X_k,Y_k,X^0, Y^0$ all have the same color. In particular, $h_0$ maps $X^0$ homeomorphically onto $Y^0$, and we can define    $$h_1|X_k \coloneqq   (g|Y_k)^{-1} \circ h_0 \circ (f|X_k).$$ 
    This is a homeomorphism from $X_k$ onto $Y_k$. 
 It is then  
    straightforward to check that these partial homeomorphisms  on the $1$-tiles 
    $X_k$  of $\CDach$ paste together to a well-defined homeomorphism
   $h_1\colon \CDach \to
    \mathcal{S}^1\cong S^2$. Moreover,  it
    follows 
    directly from the definition of $h_1$ that $h_0\circ f = g\circ h_1$,  and so we get a  commutative diagram as in  
    \eqref{Thequiv1}. 

      It remains to argue  that the homeomorphisms $h_0$ and $h_1$ are 
      isotopic rel.\
    $\post(f)$. Note that $h_0$ and $h_1$ are
orientation-preserving and agree on the set $\post(f)=\{\omega,1, \infty\}$.
Moreover,  $h_0$ maps $\CC$ (the extended line through $\omega$ and $1$) and $h_1$ maps $R_0 \cup R_4$ to the  equator of the pillow. So  basically we need to deform $\CC$ to $R_0 \cup
R_4$ while keeping $\post(f)=\{\omega,1, \infty\} = \CC \cap (R_0 \cup R_4)$
fixed to obtain the desired isotopy between $h_0$ and $h_1$. It is intuitively clear that this is possible. Since $\post(f)=3$ the existence of the desired isotopy actually follows from a general fact (see Lemma~\ref{lem:homeo}). \end{ex}

\section{The orbifold associated with  a Thurston map}
\label{sec:orbif-assoc-thurst}

An  orbifold\index{orbifold} is a space that is locally represented  as a quotient  of a model   space  by a  
 group action (see
\cite[Chapter~13]{Th}).  In our present context we are only interested in the $2$-dimensional 
case where the group actions are given by cyclic groups near  points on a surface. Then  all the relevant information is given by  a pair 
$(S, \alpha)$, where 
$S$ is a surface and $\alpha\colon S\to \widehat{\N}\coloneqq \N\cup\{\infty\}$ is a map such that the set
of points $p\in S$ with $\alpha(p)\neq 1$ is a 
discrete set in $S$, i.e., it has no limit points in $S$. 
We call such a function $\alpha$ a 
 {\em ramification function}\index{ramification!function} on $S$, and  the pair $(S, \alpha)$ an {\em orbifold}.\index{orbifold}  Each 
 number $\alpha(p)$ should be thought of as the order of an associated cyclic group  (see Section~\ref{sec:orbifolds-coverings} and in particular 
 Proposition~\ref{prop:prop_deck_trafo}~\ref{item:pi1_O_2}). 
   The  
  set $\supp(\alpha)\coloneqq \{ p\in S: \alpha(p)\ge 2\}$
  is  the 
  {\em support}\index{ramification!function!support of}\index{supp@$\supp$}  
  of $\alpha$. 
It is discrete 
in  $S$;  so  if $S$ is compact, then 
  $\supp(\alpha)$ is a finite set. 

Orbifolds are useful for the study of branched covering maps, in particular Thurston maps. In this case, the underlying surface of the orbifold is a topological $2$-sphere. Every Thurston map 
$f\: S^2 \ra S^2$ has an associated orbifold $\mathcal{O}_f=(S^2, \alpha_f)$, where  the orbifold data are 
determined  by the {\em ramification function} $\alpha_f\: S^2 \ra \widehat{\N}$ of $f$. 
We will first define 
$\alpha_f$ and  discuss  the  main properties of this function before we turn our attention  to $\mathcal{O}_f$. 

For a general Thurston map $f$ we consider its orbifold  $\mathcal{O}_f$ as a purely topological object that encodes some ramification data of $f$. In particular, the orbifold $\mathcal{O}_f$ 
does not carry  any canonical  
 geometric structure associated with $f$. This is different for {\em rational} Thurston maps $f$ defined 
 on the Riemann sphere $\CDach$. Here the additional conformal structure on $\CDach$  can be used to define the {\em canonical orbifold metric} of $f$. We will discuss this briefly at the end of this section,  and in more detail in Section~\ref{sec:expratThmaps}.

We use the notation 
$\widehat\N \coloneqq  \N\cup\{\infty\}$,\index{n2@$\widehat{\N}$}\index{$\norm{Da}$@$a\vert b$} 
and 
extend the usual order relations $<, \le , >, \ge$
 on $\N$ in the obvious way  to  $\widehat\N$.  So $a<\infty$ for $a\in \N$,
$a\le \infty$ for $a\in  \widehat\N$,  etc.   We also extend multiplication of natural numbers to $\widehat \N$ by setting $a\cdot \infty=\infty \cdot a=\infty$ for $a\in \widehat \N$. For $a,b\in \widehat \N$  we say that {\em $a$ divides $b$}, written $a| b$, if there exists $k\in \widehat \N$ such that $b=ak$. So the relation $a|b$ is 
an extension of the usual divisor relation in $\N$ with  the additional convention that every value in $\N\cup\{\infty\}$ divides $\infty$. 


Suppose $A\sub \widehat\N$ is arbitrary. Then there exists a unique  $L\in \widehat\N$,
called the {\em least common multiple} of the elements of  $A$ and denoted by 
$\lcm(A)$,\index{lcm@$\lcm$} 
with the following properties:  $a| L$ for all $a\in A$, and if
$L'\in \widehat\N$ is such that $a| L'$ for all $a\in A$, then
$L|L'$. 
It is easy to see that if $A\sub \N$ is a finite  set of natural numbers, then  $\lcm(A)\in \N$ is the 
least common multiple of the numbers in $A$ in the usual sense, and  $\lcm(A)=\infty$ otherwise.

If  $\alpha, \beta\: S^2\ra \widehat \N$ are  functions on a $2$-sphere $S^2$, then  we write 
$\alpha\leq\beta$ and $\alpha|\beta$ if $\alpha(p)\leq \beta(p)$ and $\alpha(p)|\beta(p)$ for all 
$p\in S^2$, respectively. 


\begin{definition}[Ramification function of a Thurston map]
  \label{def:weightf} Let $f\colon S^2 \to S^2$ be a Thurston map. 
  Then
  its \emph{ramification function}\index{ramification!function}\index{a f@$\alpha_f$} 
  is the map 
  $\alpha_f\: S^2 \ra \widehat \N$ defined for $p\in S^2$ as  
   \begin{equation*}\label{eq:deframffunc} 
      \alpha_f(p)=\lcm\{\deg(f^n,q) : q\in S^2,\,  n\in \N, \text{ and } f^n(q)=p \}.
    \end{equation*}
 \end{definition}

 We will see momentarily (see the remark after
 Proposition~\ref{prop:otherramprops}) that $\alpha_f$ is indeed
 a ramification function on the surface $S^2$. It admits the following characterization.
 
 \begin{prop}
  \label{prop:weightf} Let $f\: S^2\ra S^2$ be a Thurston map with  its associated ramification function $\alpha_f\: S^2 \ra \widehat \N$. Then we have:  
  \begin{enumerate}
  \item 
    \label{item:weight1} $\deg(f, q)\alpha_f(q)\,|\, \alpha_f(p)$, whenever $p,q\in S^2$ and 
    $f(q)=p$. 
    
   \item 
    \label{item:weight2} If $\beta\: S^2 \ra \widehat \N$ is any function such that 
    $\deg(f, q)\, \beta(q)\,|\,  \beta(p)$ whenever $p,q\in S^2$ and 
    $f(q)=p$, then $\alpha_f|\beta$. 
  \end{enumerate}
  Moreover, $\alpha_f$ is the unique function with the properties \ref{item:weight1}
  and \ref{item:weight2}. 
\end{prop}

Before we turn to the proof of this proposition, we point out a fact that follows from repeated application of \ref{item:weight1}. Namely, suppose that $p,q\in S^2$, $n\in \N$, and $f^n(q)=p$. 
 Let $q_k \coloneqq  f^k(q)$ for $k=0,\dots,n$. Then  
 $$ \deg(f^n, q)=\deg(f, q_0)\deg(f, q_1)\cdots \deg(f,q_{n-1})$$
 as follows from  \eqref{eq: localdegreemult}. 
 Moreover, $$\alpha_f(q_k) \deg(f, q_k)\,|\, \alpha_f(q_{k+1})$$ for $k=0, \dots, n-1$. Since 
 $q_0=q$ and  $q_n=p$,
it follows that 
$\deg(f^n, q) \alpha_f(q)\,|\,\alpha_f(p)$.

 \begin{proof}
   [Proof of Proposition~\ref{prop:weightf}] Let $f\colon S^2 \to S^2$ be a  Thurston map and $\alpha_f$ its 
ramification function  as given in Definition~\ref{def:weightf}. 

To 
 establish \ref{item:weight1} for $\alpha_f$, suppose $p,q\in S^2$
   and $f(q)=p$.  If $\alpha_f(p)=\infty$, then $\deg(f, q)
   \alpha_f(q)\,|\,\alpha_f(p)$, and there is nothing to prove. So we may
   assume that $\alpha_f(p)\in \N$.

Suppose $q'\in S^2$ is a point 
with $f^n(q')=q$ for some $n\in \N$. 
Then $f^{n+1}(q')=p$, and 
\begin{equation*}
  \deg(f^{n+1}, q')= \deg(f,q)\deg(f^n, q');   
\end{equation*}
so $\deg(f,q)\deg(f^n, q')\,|\,\alpha_f(p)$ by definition of
$\alpha_f(p)$. 
It follows that $\alpha_f(p)/\deg(f,q)$ is a natural number
that has  
$\deg(f^n, q')$ as a divisor. This implies that 
the least common multiple $\alpha_f(q)$ of all such numbers  $\deg(f^n, q')$ divides $\alpha_f(p)/\deg(f,q)$. We conclude  that 
$$\deg(f, q)
  \alpha_f(q)\,|\,\alpha_f(p), $$ and  \ref{item:weight1} follows.

Let $\beta\: S^2\ra \widehat \N$ be another function with the property \ref{item:weight1}.
Then by repeated use  of this property we see  that 
$\deg(f^n, q) \beta_f(q)\,|\, \beta(p)$ whenever $p,q\in S^2$, $n\in \N$, and $f^n(q)=p$. 
In particular, $\deg(f^n, q)| \beta(p)$. As this is true for all $q\in S^2$ with $f^n(q)=p$ for some 
$n\in \N$, we conclude $\alpha_f(p)|\beta(p)$ for all $p\in S^2$. This establishes the desired property \ref{item:weight2} of $\alpha_f$.

The uniqueness of the function $\alpha_f$ with the  properties \ref{item:weight1} and \ref{item:weight2}
is clear. 
\end{proof} 

To state other important  properties of the ramification
function, we first 
need a definition. A 
{\em critical cycle}\index{critical!cycle}
 $C$ of a Thurston map $f\:
S^2\ra S^2$ is the   orbit of a periodic critical point of $f$; so then
there exists $c\in \crit(f)$,  and $n\in \N$ with $f^n(c)=c$ such that  
$ C=\{ f^k(c): k=0, \dots, n-1\}$.

\begin{prop}
  \label{prop:otherramprops}
  Let $f\: S^2\ra S^2$ be a Thurston map with its associated ramification function $\alpha_f\: S^2 \ra
  \widehat \N$. Then for $p\in
  S^2$ we have:
  
  \begin{enumerate}
  \item 
    \label{item:rami_postf}
    $\alpha_f(p)\ge 2$ if and only if $p\in \post(f)$.  
    
  \item 
    \label{item:rami_infty}
    $\alpha_f(p)=\infty$ if and only if $p$ is contained in a critical
     cycle of $f$.    
   \end{enumerate}
 \end{prop}

In particular, this implies that $\supp(\alpha_f)=\{p\in S^2:\alpha_f(p)\ge 2\}=\post(f)$ is a finite set and so
$\alpha_f$ is really a ramification function on the surface $S^2$ (as defined in the beginning of this section).  

\begin{proof} \ref{item:rami_postf}   
  If $p\in S^2\setminus \post(f)$, then $\deg(f^n, q)=1$ whenever
  $q\in S^2$, $n\in \N$, and $f^n(q)=p$, because none of the
  points $f^k(q)$, $k=0, \dots, n-1$, can be a critical point of
  $f$. Hence $\alpha_f(p)=1$ by definition of $\alpha_f$.
 
  If $p\in \post(f)$, then there exist $c\in \crit(f)$ and
  $n\in \N$ such that $f^n(c)=p$.  Then $\deg(f,c)\ge 2$ which
  implies that $\deg(f^n, c)\ge 2$. By definition of $\alpha_f$
  we have $\deg(f^n,c)|\alpha_f(p)$; so $\alpha_f(p)\ge 2$.

  \smallskip
  \ref{item:rami_infty}
  If $p$ is contained in a critical cycle, then there exists a
  periodic critical point $c$ of $f$ such that $f^n(c)=p$ for
  some $n\in \N_0$. There exists $k\in \N$ such that $f^k(c)=c$,
  and so $f^{n+km}(c)=p$ for all $m\in \N$. Then the numbers
  $\deg(f^{n+km}, c)\ge \deg(f,c)^m\ge 2^m $ divide $\alpha_f(p)$
  for all $m\in \N$. This is only possible if
  $\alpha_f(p)=\infty$.

Conversely, suppose $p\in S^2$ and $\alpha_f(p)=\infty$. Then by definition of $\alpha_f$ the set 
$\{ \deg(f^n, q): q\in S^2, \, n\in \N, f^n(q)=p\}$ is unbounded. In particular, there exist
$q\in \N$ and  $n\in \N$ with $f^n(q)=p$ such that 
$\deg(f^n, q)>M$, where 
$$M \coloneqq  \prod_{c\in \crit(f)}\deg(f, c)\in \N. $$
Let $q_k=f^k(q)$ for $k=0, \dots, n-1$.
Then  
$$M< \deg(f^n, q)= \prod_{k=0}^{n-1} \deg(f, q_{k}). $$
Since  $\deg(f, q_{k})>1$ only if $q_k$ is a critical point and since $ \deg(f^n, q)> M$, there exists 
a critical point $c$ of $f$ that appears at least twice in the list $q_0, \dots , q_{n-1}$. Then $c$ is periodic and $p$ belongs to the orbit of $c$. Hence $p$ is an element of a critical cycle of $f$. 
\end{proof}

We can now define the orbifold of a Thurston map from the  point of view explained in the beginning of this  section. 
  
  \begin{definition}[Orbifold of a Thurston map]
  \label{def:orbifold_f}
  Let $f\colon S^2 \to S^2$ be a Thurston map. The
  \emph{orbifold}\index{orbifold}\index{Thurston map!orbifold of}\index{orbifold!of Thurston map}\index{O f@$\OC_f$} 
  associated with $f$ is the pair 
  $\OC_f\coloneqq  (S^2, \alpha_f)$, where $\alpha_f\colon S^2 \to \widehat \N$ is the ramification
  function of $f$.
  \end{definition}
   If $\mathcal{O}=(S^2, \alpha)$ is an orbifold, then 
  the 
\emph{Euler characteristic}\index{orbifold!Euler characteristic}\index{Euler characteristic!of orbifold}
of  $\OC$ is defined as
  \begin{equation}\label{def: Eulerchar}  
    \chi(\OC)= 2 - \sum_{p\in S^2} \left(1 -
      \frac{1}{\alpha(p)}\right). 
  \end{equation}
 Here and elsewhere we use the convention that $a/\infty=0$ for $a\in \N$. The sum in \eqref{def: Eulerchar} is really a finite sum, where only the points in $\supp(\alpha)$ give a  non-zero contribution, but it is convenient to write 
it  (and similar sums below) in this form.  A geometric interpretation of $ \chi(\OC)$ is given in Section~\ref{sec:orbifolds-coverings}.

   We call $\OC$ 
\emph{parabolic}\index{orbifold!parabolic|textbf}\index{parabolic!orbifold|textbf}\index{Thurston map!parabolic} 
if  $\chi(\OC)=0$, and 
\emph{hyperbolic}\index{orbifold!hyperbolic|textbf}\index{hyperbolic!orbifold|textbf}\index{Thurston map!hyperbolic} 
if $\chi(\OC)<0$.  The orbifold $\OC_f$ of a Thurston map    $f$  is always parabolic or hyperbolic.
 In order to show this,  we first need a lemma.
 
 \begin{lemma}\label{lem:orbcov} Let $\alpha, \alpha'\:S^2 \ra \widehat {\N}$ be ramification functions on a $2$-sphere $S^2$, and let $f\: S^2\ra S^2$ be a branched covering map satisfying $\deg_f(p)\cdot \alpha'(p)=\alpha(f(p))$
 for all $p\in S^2$. Then for the orbifolds $\mathcal{O}'=(S^2, \alpha')$ and $\mathcal{O}=(S^2, \alpha)$  we have $\chi(\mathcal{O}')=\deg(f)\cdot\chi(\mathcal{O})$.  
\end{lemma} 

This statement 
is an orbifold version of the  
identity 
$\chi(X)=\deg(f) \cdot  \chi(Y)$ for covering maps $f\: X\ra Y$ between surfaces (or more general spaces) $X$ and $Y$.  For compact surfaces this is a special case of the Riemann-Hurwitz formula 
\eqref{eq:Riemann-Hurwitz}.

\begin{proof} 

Let $d\coloneqq  \deg(f)$. 
Then for each point $q\in S^2$ we have 
$$\sum_{p\in f^{-1}(q)}\deg_f(p)=d. $$ 
So  the Riemann-Hurwitz formula
\eqref{eq:Riemann-Hurwitz} implies that  
\begin{align*} 2-\chi(\mathcal{O}') &=\sum_{p\in S^2}
 \left(1 -\frac{1}{\alpha'(p)}\right) =  \sum_{p\in S^2} \left(1 - \frac{\deg_f(p)} {\alpha(f(p))}\right) \\
 &= \sum_{p\in S^2} (1-\deg_f(p))+ \sum_{p\in S^2} \left(\deg_f(p) - \frac{\deg_f(p)} {\alpha(f(p))}\right) \\ 
 &= 2-2d+\sum_{q\in S^2} \sum_{p\in f^{-1}(q)} \left(\deg_f(p) - \frac{\deg_f(p)} {\alpha(f(p))}\right)\\
&= 2-2d+ d \sum_{q\in S^2} \left(1 - \frac{1}{\alpha(q)}\right)\\
      &= 2-2d+d (2-\chi(\mathcal{O})) = 2-d \chi(\mathcal{O}). 
         \end{align*}
         Note that each sum in this computation actually  has 
         only finitely many non-zero terms. 
         The claim follows. 
\end{proof}

\begin{prop} 
  \label{prop:chiO_f_nonneg}
  Let $f\: S^2\ra S^2$ be a Thurston map,  and $\alpha_f$ be its associated
  ramification  function. Then  
  \begin{equation}\label{eq:parahyp}
   \chi(\mathcal{O}_f)=2- \sum_{p\in S^2} \biggl(1-\frac{1}{\alpha_f(p)}\biggr) \le 0.
  \end{equation}
 Here we have equality if and only if $\deg_f(p)\cdot \alpha_f(p)=\alpha_f(f(p))$ for all $p\in S^2$. 
\end{prop}

\begin{proof} 
For $p\in S^2$ we define 
$$ \alpha'(p)=
\left\{ \begin{array} {cl} \alpha_f(f(p))/\deg_f(p) &\text{if   $\alpha_f(f(p))<\infty$,} \\
\infty&
 \text{if   $\alpha_f(f(p))=\infty$.}  \end{array}
 \right.
 $$
 By Proposition~\ref{prop:weightf} the function $\alpha'$ takes values in
 $\widehat {\N}$ and we have $\alpha'\ge \alpha_f$. Moreover, 
 $\{p\in S^2:  \alpha'(p)\ge 2\} \sub f^{-1}(\post(f))$ is a finite set, and so $\alpha'$ is a ramification function on $S^2$. It satisfies 
 $\deg_f(p)\cdot \alpha'(p)= \alpha_f(f(p))$ for all $p\in S^2$. 
 Hence by Lemma~\ref{lem:orbcov} we have $\chi(\mathcal{O}')=d \chi(\mathcal{O}_f)$, where $d=\deg(f)\ge 2$ and $\mathcal{O}'=(S^2, \alpha')$. 
 
 On the other hand, the fact that $\alpha'\ge \alpha_f$ and the definition of the Euler characteristic of an orbifold imply that $\chi(\mathcal{O}')\le \chi(\mathcal{O}_f)$. Hence 
 $$(d-1)\chi(\mathcal{O}_f)= \chi(\mathcal{O}')-\chi(\mathcal{O}_f)\le 0. $$
 We conclude that   $\chi(\mathcal{O}_f)\le 0$. Here we have equality if and only if 
 $\chi(\mathcal{O}')=\chi(\mathcal{O}_f)$ which in view of $\alpha'\ge \alpha_f$  is in turn  equivalent to $\alpha'=\alpha_f$.  This last condition is the same as the requirement that  
 $\deg_f(p)\cdot \alpha_f(p)=\alpha_f(f(p))$ for all $p\in S^2$. The statement follows.  
 \end{proof}

By the previous proposition  the orbifold of a Thurston map $f\: S^2\ra S^2$  is parabolic or hyperbolic. 

Let $\mathcal{O}= (S^2, \alpha)$ be an orbifold. If we  label the finitely many points $p_1, \dots, p_k$ 
in $\supp(\alpha)$ so that 
$2\le \alpha(p_1)\le \dots \le \alpha(p_k)$, then the $k$-tuple $$(\alpha(p_1), \dots,  \alpha(p_k))$$ is called the \emph{signature}\index{orbifold!signature}\index{signature} of
  $\OC.$ The 
{\em signature of a Thurston map}\index{Thurston map!signature} 
  $f\: S^2\ra S^2$ is the signature of its orbifold 
  $\mathcal{O}_f=(S^2, \alpha_f)$. 
   Note  that  in this case the  support of the ramification
   function $\alpha_f$ consists  precisely of the points in
   $\post(f)$ (see  Proposition~\ref{prop:otherramprops}), and so
   the signature of $f$ is determined by the restriction of
   $\alpha_f$ to $\post(f)$.  
        
   The Latt\`{e}s  map $g$ from Section~\ref{sec:Lattes} has
  signature 
$(2,2,2,2)$, and accordingly its associated  orbifold $\OC_g$ is parabolic.

We  record the following immediate consequence 
of Proposition~\ref{eq:parahyp}. 
\begin{cor}
  \label{cor:post012}
 If $f\colon S^2\to S^2$ is a Thurston map, then $\#\post(f) \geq
 2$.  
 Moreover,  $\#\post(f) =2$ if and only if
  the signature of $f$ is $(\infty,\infty)$.
\end{cor}

\begin{proof}
 If   $\#\post(f) \in \{0,1\}$,  then for the  orbifold $\mathcal{O}_f$ of $f$ we have  $\chi(\mathcal{O}_f) >0$ by \eqref{def:
    Eulerchar} and 
    Proposition~\ref{prop:otherramprops}~\ref{item:rami_postf}. This  contradicts
  Proposition~\ref{prop:chiO_f_nonneg}. 
  
  Similarly, since $\chi(\mathcal{O}_f)\leq 0$, we have  $\#\post(f)
  =2$   if and only if the
  signature of $f$ is $(\infty,\infty)$. 
\end{proof}

We will see later that when $\#\post(f)=2$, the map $f$ is
 Thurston equivalent to the map given by $z\mapsto z^n$ on
$\CDach$, where $n\in \Z\setminus\{-1,0,1\}$ 
  (Proposition~\ref{prop:post2};  see also
  Lemma~\ref{lem:postf-2}).

Parabolicity of the orbifold of a Thurston map admits  various cha\-racterizations. 

\begin{prop}[Thurston maps with parabolic orbifold]
  \label{prop:parabolicOf}  
  \index{orbifold!parabolic} 
  \index{Thurston map!parabolic}
  \index{parabolic!orbifold}


  Let $f\colon S^2 \to S^2$ be a Thurston map. Then the  following
  conditions are equivalent:
  \begin{enumerate}
  \item 
    \label{item:Of_para1}
    $\OC_f$ is parabolic. 
    
  \item
    \label{item:Of_para2}
    The signature of $\OC_f$ is 
    $$(\infty,\infty),\, (2,2,
    \infty),\, (2,2,2,2),\, (2,4,4),\,(3,3,3),\,  \text{or } (2,3,6). $$ 
    
  \item
    \label{item:Of_para3}
    The ramification function $\alpha_f\colon S^2 \to \widehat \N$ satisfies  
    $$\deg_f(p)\cdot \alpha_f(p)=\alpha_f(f(p))$$
    for all   $p\in S^2$.   \end{enumerate}
\end{prop}

Rational Thurston maps with parabolic orbifolds are
investigated in Chapter~\ref{cha:lattes-lattes-type} and
Section~\ref{sec:parab-thurst-polyn}. Another characterization of
Thurston maps with parabolic 
orbi\-folds 
is given in 
Lemma~\ref{lem:parabolic_characterization}.

\begin{proof} \ref{item:Of_para1} $\Leftrightarrow$ \ref{item:Of_para2} If $\OC_f$ has one of the signatures listed in \ref{item:Of_para2}, then $\chi(\OC_f)=0$, and so $\OC_f$ is parabolic. 
Conversely, if $\chi(\OC_f)=0$ then one first notes that $f$ can have at most four postcritical points.
 Exhausting  all combinatorial possibilities, we are led  to the signatures in  \ref{item:Of_para2}.

  \smallskip
 \ref{item:Of_para1} $\Leftrightarrow$   \ref{item:Of_para3}  This immediately follows from the second part of Proposition~\ref{prop:chiO_f_nonneg}. 
  \end{proof}

It is an elementary fact that  the signatures of equivalent Thurston maps are  the same. 

\begin{prop}\label{prop:Thequivsamesig}
  Two Thurston maps that are Thurston equivalent have the same
  signatures.  
\end{prop} 

\begin{proof} 
Let $f \colon S^2\ra S^2$ 
and $g\:\widehat S^2 \ra \widehat S^2$ be two Thurston maps
on $2$-spheres $S^2$ and $\widehat S^2$, and suppose they are Thurston equivalent. Then there exist homeomorphisms $h_0,h_1\: S^2 \ra \widehat S^2$ as in Definition~\ref{def:Thequiv}.  

Let $\alpha_f\: S^2\ra \widehat \N$ and $\alpha_g\: \widehat S^2\ra \widehat \N$  be the ramification 
functions of $f$ and $g$, respectively. The claim will follow if we can show that 
$\alpha_f=\alpha_g \circ h_0$. 

To establish this identity,    
define $\nu \coloneqq \alpha_g \circ h_0$. Since $\alpha_g$ is supported on $\post(g)$, $h_0|\post(f)=h_1|\post(f)$, and $h_0(\post(f))=h_1(\post(f))=\post(g)$ (see Lemma~\ref{lem:T-eq_crit_post}), we have 
$\nu =\alpha_g \circ h_0=\alpha_g\circ h_1$. 

Now let  $p\in S^2$ be  arbitrary, and define $\widehat p\coloneqq h_1(p)$. By what we have seen, 
$\nu(p)=\alpha_g(h_1(p))=\alpha_g(\widehat p\,)$; moreover, the relation $h_0\circ f=g\circ h_1$ implies that $\deg(f,p)=\deg(g, \widehat p\,)$ and $h_0(f(p))=g(\widehat p\,)$. 
Hence 
$$ \nu(p)\cdot \deg(f,p)= \alpha_g(\widehat p\,)\cdot \deg(g,\widehat p\,)
$$ 
divides 
$$ \alpha_g(g(\widehat p\,))=  \alpha_g(h_0(f(p)))= \nu(f(p)).$$
 Proposition~\ref{prop:weightf}~\ref{item:weight2} implies  that $\alpha_f$ divides $\nu =\alpha_g \circ h_0$. If we reverse the roles of 
$f$ and $g$, then a similar argument shows that $ \alpha_g|\alpha_f\circ h_0^{-1}$, or equivalently, $\alpha_g\circ h_0| \alpha_f$. So    $\alpha_f= \alpha_g\circ h_0$ as desired.
\end{proof} 

The ramification function and hence  the signature of a Thurston map do not change if we pass to any of its iterates.

\begin{prop}
  \label{prop:itsamesig}
  \index{Thurston map!iterate of}
  \index{iterate of Thurston map}
  \index{F f@$F=f^n$}
Let $f\: S^2 \ra S^2$ be a Thurston map. Then $\alpha_f=\alpha_{f^n}$ for each $n\in \N$. 
\end{prop} 

\begin{proof} Fix $n\in \N$, and let $F=f^n$. If $p\in S^2$, $k\in \N$,  and $q\in F^{-k}(p)$, then $p=F^k(q)=f^{nk}(q)$, and so 
$$\deg(F^k, q)=\deg(f^{nk}, q)|\alpha_f(p) $$
by definition of $\alpha_f$ (see Definition~\ref{def:weightf}). Since this is true for all $k\in \N$ and $q\in F^{-k}(p)$, this in turn implies $\alpha_F(p)|\alpha_f(p)$ by definition of $\alpha_F$;  so $\alpha_F|\alpha_f$. 

On the other hand, suppose  $p\in S^2$, and let $k\in \N$ and $q\in f^{-k}(p)$  be arbitrary. 
Then there exist $l,m\in \N$ such that $F^m=f^k\circ f^l$. We can find a point $q'\in S^2$ such that 
$f^l(q')=q$. Then 
$$ \deg(F^m, q')=\deg(f^k, q)\cdot \deg(f^l, q'), $$   
and so $\deg(f^k, q)| \deg(F^m, q')$. Since $F^m(q')=f^k(q)=p$, we have $\deg(F^m, q')|\alpha_F(p)$ by definition of $\alpha_F$, which  implies  $\deg(f^k, q)| \alpha_F(p)$. Since 
$k\in \N$ and $q\in f^{-k}(p)$ were arbitrary, we have  $\alpha_f(p) |\alpha_F(p)$ by definition of $\alpha_f$. We conclude that $\alpha_f|\alpha_F$;  but    we have seen above that $\alpha_F|\alpha_f$, and so  $\alpha_f=\alpha_F$ as desired. \end{proof} 

We finish  this section with a brief discussion of the {\em canonical orbifold metric} associated with  an orbifold $\mathcal{O}=(\CDach, \alpha)$ whose underlying surface is the Riemann sphere
$\CDach$. Here we assume that 
$\mathcal{O}$ is parabolic or hyperbolic, the only cases relevant for Thurston maps. 
With  the ramification function  $\alpha$ understood, we use the notation 
$$\CDach_0\coloneqq \CDach\setminus \{ z\in \CDach: \alpha(z)=\infty\}. $$ 
So  $\CDach_0$ is the Riemann sphere with each  
{\em puncture}\index{puncture}\index{orbifold!puncture of}     
of the orbifold $\mathcal{O}=(\CDach, \alpha)$ (i.e.,  a point   $p\in\CDach$ with $\alpha(p)=\infty$) removed. 

We set $X=\C$ or $X=\D$ depending on whether  $\mathcal{O}$ is parabolic or hyperbolic.
Then there exists a holomorphic branched covering map $\Theta\: X\ra \Cdach_0$ 
such that 
$$ \deg(\Theta, z)=\alpha(\Theta(z))$$ 
for each $z\in X$. The map $\Theta$ is unique up to a precomposition with a biholomorphism of $X$ and called the {\em  universal covering map} of the orbifold $\mathcal{O}$. 
These facts are discussed in detail in Section~\ref{sec:orbifolds-coverings}. 

We equip $X$ with its natural metric $d_0$, namely the Euclidean
metric if $X=\C$ and the hyperbolic metric if $X=\D$. Then one
can show that 
there exists a  metric $\om$ on $\CDach_0$, 
called
the {\em canonical  orbifold metric}\index{canonical orbifold!metric}\index{metric!canonical orbifold}\index{o@$\omega$}\index{orbifold!canonical metric}\index{push-forward!of metric!by orbifold covering map}
 of $\mathcal{O}$, which is given by   
$$ \om(p,q)=\inf\{d_0(z,w): z\in \Theta^{-1}(p), \, w\in \Theta^{-1}(q)\}$$
for $p,q\in \CDach_0$. Since the  universal covering map $\Theta$ of $\mathcal{O}$ is essentially 
unique, the metric $\om$ is uniquely determined if $\mathcal{O}$ is hyperbolic and 
uniquely determined up to a scaling factor if $\mathcal{O}$ is parabolic. Note that this conclusion strongly relies on holomorphicity of the maps involved; namely, it follows  from the fact that a biholomorphism of $\D$ preserves the hyperbolic metric and that a  biholomorphism of $\C$
is a Euclidean similarity and so scales Euclidean distances by a fixed factor.

 If we equip $X$ with the metric $d_0$ and $\CDach_0$ with the
 metric $\om$, then $\Theta$ is a 
path isometry\index{path!isometry}
in the sense that 
$$ \length_{\om}(\Theta\circ \beta)=\length_{d_0}(\beta)$$ 
for all paths $\beta$ in $X$. This property characterizes the metric $\om$; so roughly speaking, one can say that the canonical orbifold metric $\om$ of $\mathcal{O}$ is obtained by pushing forward the natural metric $d_0$ on $X$ by the   universal covering map $\Theta$ to $\CDach_0$.

 A point $p\in \CDach$ with $2\le\alpha(p)<\infty$ is called a
 {\em conical singularity} or 
{\em cone point}\index{cone!point}\index{cone!point}\index{orbifold!cone point of}\index{conical singularity}\index{singularity!conical} 
of the orbifold $(\CDach, \alpha)$. At such a point, $(\CDach_0, \om)$ is locally isometric to a  (Euclidean or hyperbolic) cone with cone angle $2\pi/\alpha(p)$ at $p$. At all other points,  $(\CDach_0, \om)$ is locally isometric to the model space $X$. 

If $f\: \CDach \ra \CDach$ is a rational Thurston map, then its orbifold $\mathcal{O}_f=(\Cdach, \alpha_f)$ is parabolic or hyperbolic and so the preceding discussion applies. We call the metric $\om=\om_f$ for the orbifold $\mathcal{O}_f$, the {\em canonical orbifold metric}
of $f$. Note that in the parabolic case, it is only unique up to scaling, but often this ambiguity does not matter. Since the essential uniqueness of $\om$ strongly relies on the holomorphicity assumption, one cannot define a similar canonical metric for a general Thurston map $f\: S^2\ra S^2$ defined on a topological $2$-sphere $S^2$ with no conformal  structure (note though that sometimes 
it is useful to identify $S^2$ with $\CDach$ and pick a
suitable orbifold metric on $\CDach$; see the proof of Proposition~\ref{prop:expLattType}, for example). 

 For a detailed discussion of the universal orbifold metric (and also the definition of a natural associated measure, the {\em canonical orbifold measure}), see Section~\ref{sec:expratThmaps}. 

\section{Thurston's characterization of rational maps}
\label{sec:thurst-class-rati}

Thurston's  criterion when a Thurston map is 
equivalent to a rational map is an important theorem in complex dynamics.
To formulate this statement, we need  some definitions.

\pagebreak

\renewcommand{\labelitemi}{$\ast$}
\begin{definition}[Invariant multicurves]
  \label{def:multicurve}
  Let $f\colon S^2\to S^2$ be a Thurston map.
  
  \begin{enumerate}
  
  \item 
    \label{item:multi_peri}
    A Jordan curve   $\gamma\subset S^2\setminus\post(f)$ is called 
\emph{non-peripheral}\index{multicurve!non-peripheral}\index{non-peripheral}
    if each
    of the two components of $S^2\setminus \gamma$ contains at least two points
    from $\post(f)$, and is called  
    \emph{peripheral}\index{peripheral}
    otherwise.
  
  \item 
    \label{item:def_multicurve}
    A 
    \emph{multicurve}\index{multicurve}
     is a non-empty finite set of non-peripheral Jordan curves in
    $S^2\setminus \post(f)$ that are pairwise disjoint and
    pairwise non-isotopic rel.~$\post(f)$.

  \item 
    \label{item:def_inv_multi}
    A multicurve $\Gamma$ is called 
    {\em $f$-invariant}\index{multicurve!invariant}\index{invariant!multicurve} 
    (or simply {\em invariant} if $f$ is understood)
    if each non-peripheral component of the preimage
    $f^{-1}(\gamma)$ of a curve $\ga\in \Gamma$ is isotopic rel.\
    $\post(f)$ to a curve $\ga'\in \Gamma$.
  \end{enumerate}
\end{definition}

Note that if $\ga \sub S^2$ is a Jordan curve, then by the
Sch\"onflies theorem $S^2\setminus \ga$ has precisely two
components, each of which is a topological disk.
 
If $\#\post(f)\le 3$ every Jordan curve in $S^2\setminus \post(f)$
is peripheral; so there are no multicurves in this case.  If
$\#\post(f)=4$ then every multicurve consists of a single Jordan
curve.

Recall (see Section~\ref{sec:thurston-equivalence}) that we call 
that two Jordan curves
$ \ga,\ga' \sub S^2$   
isotopic rel.~$\post(f)$ if there exists an isotopy
$H\: S^2\times I\ra S^2$ rel.\ $\post(f)$ such that
$H_0=\id_{S^2}$ and $H_1(\ga)=\ga'$. 
In \ref{item:def_inv_multi} we  implicitly used
that if $\ga \sub S^2\setminus \post(f)$ is a Jordan curve, then
each component $\sigma$ of $f^{-1}(\ga)$ is also a Jordan curve
in $S^2\setminus \post(f)$. Essentially, this follows from the
fact that a suitable branch of $f^{-1}$ gives a local
homeomorphism of $\ga$ onto $\sigma$.

Two Jordan curves $\gamma, \gamma'\sub  S^2\setminus \post(f)$
are isotopic rel.\ $\post(f)$ if and only if $\gamma$ and $\gamma'$ are
homotopic in $S^2\setminus \post(f)$. This means that there is
a homotopy $K\colon S^2\setminus\post(f) \times I \to
S^2\setminus\post(f)$ such that $K_0=\id_{S^2\setminus \post(f)}$
and $K_1(\gamma)= \gamma'$ (see \cite{Ep66} for a proof of this
fact).

Suppose  $\Gamma=\{\gamma_1, \dots , \gamma_n\}$ is  an invariant multicurve
for a given  Thurston map $f\: S^2 \ra S^2$. 
 Then  one can  associate an $(n\times n)$-matrix with $f$ and $\Gamma$ as follows. 
Fix $i,j\in \{1,\dots , n\}$, and let $\sigma_1,
\dots , \sigma_k$ be the components of $f^{-1}(\gamma_j)$ that are
isotopic to $\gamma_i$ rel.\ $\post(f)$ (here $k=k(i,j)\in \N_0$). Then each  set $\sigma_l$ is a Jordan curve in $S^2\setminus \post(f)$, and $f|\sigma_l$ is a covering map 
of $\sigma_l$ onto  $\gamma_j$. Let
\begin{equation*}
  d_{i,j,l}\coloneqq  \deg(f|\sigma_l)
\end{equation*}
be the (unsigned) topological degree of this map (in this case, this is just the number of preimages of each point $p\in \gamma_j$ under the map $f|\sigma_l$).   Then the 
\emph{Thurston matrix}
\index{Thurston!matrix}  
$A= A(f,\Gamma)= (a_{ij})$ is the matrix  with non-negative entries
\begin{equation*}
  a_{ij} = \sum_{l=1}^{k(i,j)} \frac{1}{d_{i,j,l}}
\end{equation*}
for  $i,j\in \{1,\dots, n\}$; if $k(i,j)=0$, then the sum  is interpreted as the empty sum, in which case $a_{ij}=0$.  

A \emph{Thurston obstruction}
\index{Thurston!obstruction} 
for a
Thurston map $f$ is an invariant multicurve $\Gamma$ such that the
spectral radius (which is the largest eigenvalue by the
Perron-Frobenius theorem) of the Thurston matrix $A(f, \Gamma)$
is $\geq 1$.

With these definitions Thurston's criterion can be formulated as follows.

\begin{theorem}[Thurston's characterization  of rational maps]
  \label{thm:Thurston}
  Let $f\colon S^2\to S^2$ be a Thurston map with a hyperbolic
  orbifold. Then $f$ is Thurston equivalent to a rational map if and
  only if there exists no Thurston obstruction for $f$. 
\end{theorem}
\index{Thurston map!rational}
The proof can be found in \cite{DH}, see also \cite[Theo\-rem 10.1.14]{HuTeich2}.  
We will not  use this theorem in any essential way, and included its statement for general background and to put our work into context.  

A Thurston map $f$  with a parabolic orbifold is  not covered  by Theorem~\ref{thm:Thurston}. In this case, the map   has at most four postcritical points. If $\#\post(f) \le 3$, then $f$ is 
always equivalent to a rational map (see  Proposition~\ref{prop:post2} and Theorem~\ref{thm:3postrat}~\ref{item:3post_rat}). If $\#\post(f) =4$ and $f$ has a  parabolic orbifold, then the signature  of $f$ is $(2,2,2,2)$. A criterion when such a map is  equivalent  to a rational map can be derived  
from Proposition~\ref{prop:noperparaLTM} in combination with 
Theorem~\ref{thm:Thurston_para}. 

In Theorem~\ref{thm:S2vsf}~\ref{item:S2qsphere} we give a criterion for an expanding Thurston
map to be topologically conjugate to a rational map. The proof does not use Thurston's theorem.

 \begin{figure}
  \centering
  \begin{overpic}
    [width=12cm, 
    tics=10]{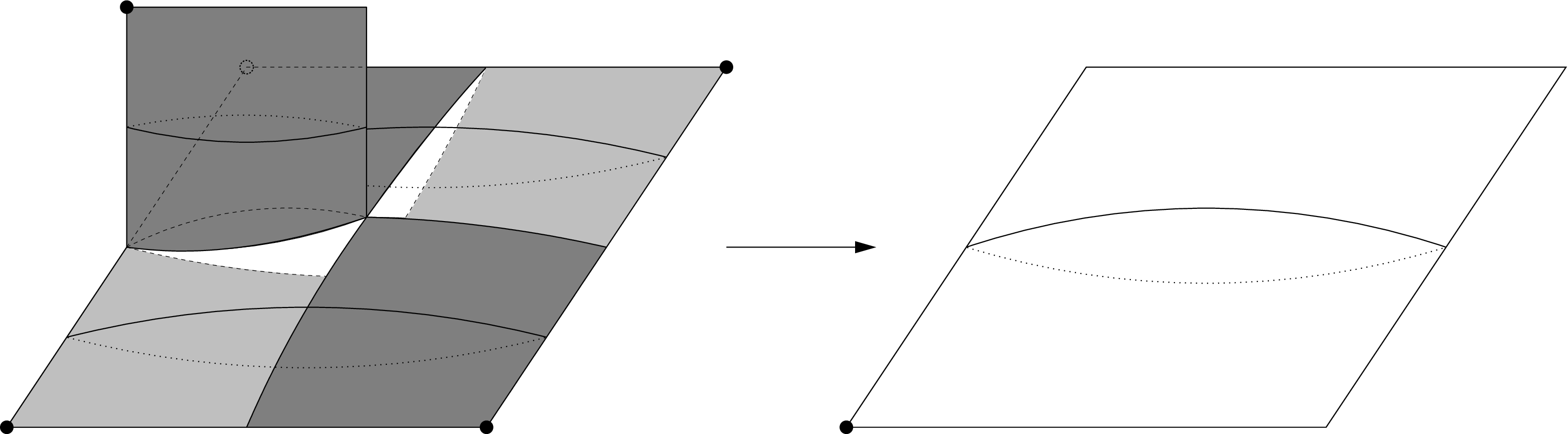}
    \put(50,14){${f}$}
    \put(94,11){$\gamma$}
    \put(36,5){$\sigma_1$}
    \put(44,17){$\sigma_2$}
    \put(4.4,19.6){$\sigma_3$}
  \end{overpic}
  \caption{An obstructed map.}
  \label{fig:obstr_map}
\end{figure}
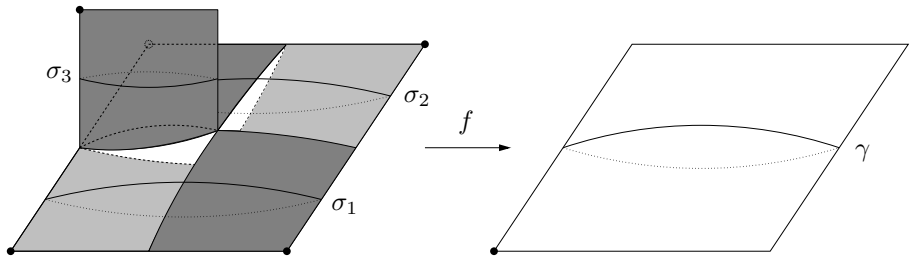

\begin{ex}\label{ex:obstructedThmap}
To illustrate Theorem~\ref{thm:Thurston} and the concepts we introduced for its formulation,  we consider the Thurston map $f=h$ constructed in 
Section~\ref{sec:int-frac-sph}.    

 Recall that
$\#\post(f)=4$, and that the postcritical points of $f$ are given by the vertices  of the pillow on the right in Figure~\ref{fig:obstr_map}.  The signature of the orbifold $\mathcal{O}_f$ of $f$ is $(2,6,6,6)$, and  so $\mathcal{O}_f$ is hyperbolic. 

Let $\gamma$ be the Jordan 
curve indicated on the 
right in  Figure~\ref{fig:obstr_map}. The preimage $f^{-1}(\ga)$ of $\ga$ 
has three components $\sigma_1$, $\sigma_2$, $\sigma_3$ indicated on  the
left in Figure~\ref{fig:obstr_map}. The Jordan curves $\sigma_1$ and 
$\sigma_2$ are non-peripheral, while $\sigma_3$ is peripheral.  Both curves
$\sigma_1$ and $\sigma_2$ are isotopic to $\gamma$ rel.\ $\post(f)$. Thus
$\Gamma=\{\gamma\}$ is an invariant multicurve.

The degree of 
$f|\sigma_l\colon \sigma_{l} \ra \gamma$ is $2$ for $l=1,2$;  so  the
Thurston matrix $A(f,\Gamma)$, which is a $1\times 1$-matrix, has the  single entry  $1/2+1/2=1$.

It follows that the spectral radius of  $A(f,\Gamma)$ is equal to $1$. Hence $\Gamma$ is a  Thurston obstruction for   $f$,  and $f$ is not
Thurston equivalent to a rational map by Thurston's criterion. 
\end{ex}  

Theorem~\ref{thm:Thurston} can be interpreted as a condition  for the existence of a Thurston equivalence. It is complemented by the following statement which is essentially a 
 uniqueness statement  for Thurston equivalences.
 
\begin{theorem}[Thurston's uniqueness theorem]
  \label{thm:Thurstonuniqness}
  Let $f,g\: \CDach \ra \CDach$ be two rational Thurston maps with  hyperbolic orbifolds, and suppose 
  that there are two orientation-preserving homeomorphisms $h_0, h_1\: \CDach \ra \CDach$ that are isotopic rel.~$\post(f)$ and satisfy 
  $h_0\circ f=g\circ h_1$. Then there exists a conformal homeomorphism    
  $\varphi\: \CDach \ra \CDach$ that is isotopic to $h_0$ and $h_1$ 
  rel.~$\post(f)$   and satisfies $\varphi\circ f= g\circ \varphi$.  
  \end{theorem}
A conformal homeomorphism on $\CDach$ is of course a M\"obius transformation. So in particular,  if two rational Thurston maps with hyperbolic orbifolds  are 
orientation-preserving Thurston equivalent (see the discussion after Definition~\ref{def:Thequiv}), then they are conjugate by a M\"obius transformation.

Theorem~\ref{thm:Thurstonuniqness} is contained in  \cite{DH}; there it is not formulated explicitly, 
but it can be easily derived from the considerations in this paper.

\ifthenelse{\boolean{singlechapter}}{

%


\chapter{Latt\`{e}s maps}
\label{cha:lattes-lattes-type}

 A \emph{Latt\`{e}s map}\index{Latt\`{e}s map} is a
rational Thurston map that is        
expanding and has a parabolic orbifold. There are  other equivalent ways to characterize these maps. For example,   a Latt\`{e}s map is a quotient of a
holomorphic automorphism of the
complex plane by the action of a crystallographic group or a quotient of a holomorphic endomorphism on  a complex
torus. We will explain this more precisely below.  

These maps play a special role in the theory. On the one hand, they are
very easy to construct and visualize, and  provide a convenient class of
examples (one was given in Section~\ref{sec:Lattes}).
 On the other hand, they often show exceptional behavior compared 
to generic rational
Thurston maps (that are expanding). This  is already apparent in Thurston's
characterization of rational maps (Theorem~\ref{thm:Thurston}). In general, Latt\`es maps are distinguished  among  typical   rational Thurston maps in terms 
of metric geometry
(Theorem~\ref{thm:S2vsf}~\ref{item:S2Lattes}), by their
measure-theoretic properties
(Theorem~\ref{thm:abscontimpLattes}), 
or by their 
``combinatorial expansion rate'' (Theorem~\ref{thm:Qianthm}). 
These statements  are among the main results of this work and so 
we will take a closer look at these maps. 
  We will also define the related class of 
\emph{Latt\`{e}s-type} maps. These are
Thurston maps with a parabolic orbifold and no periodic critical
points, but they are not necessarily (equivalent to) rational maps. 

Some aspects of a thorough treatment  of the underlying theory are
rather technical. As we do not want to overburden the reader with details at this point,  we will rely on various results that are more fully developed in the appendix. 

To motivate our definition  of Latt\`es maps in terms  of three equivalent conditions, we will now consider a specific example. For precise definitions of the 
 terminology in the ensuing discussion we refer to the beginning of Section~\ref{sec:cryst-groups-latt}. 
 
  Let $f$ be the map from Section~\ref{sec:Lattes} (there denoted by $g$). Then   $f$ is a rational  Thurston 
 map that is expanding, or equivalently, has no periodic critical points. Its orbifold has signature $(2,2,2,2)$, and is hence parabolic.

The map $f$ is   a quotient of the automorphism 
$A\colon \C\to \C$, $z\mapsto A(z)=2z$,  on $\C$ 
 by a holomorphic map
$\Theta\: \C \ra \CDach$ in the sense that the  following  diagram commutes:
\begin{equation}
  \label{eq:def_lattes2}
  \xymatrix{
    \C \ar[r]^{A}\ar[d]_{\Theta} & \C\ar[d]^{\Theta}
    \\
    \CDach \ar[r]^{f} & \CDach\rlap{.}
    }
\end{equation}
  Here $\Theta$ is essentially  a Weierstrass $\wp$-function  for the lattice $\Gamma= \Z\oplus \Z\iu$
  (see Section~\ref{sec:paraorbcov} for the definition of $\wp$ and a related discussion).  Note that 
 $\Theta(z) =\Theta(w)$ for  $z,w\in \C$ if and only if
$w= \pm z + m + n\iu$ with  $m,n\in \Z$. This last condition can most conveniently be expressed in terms of an action of a {\em crystallographic group} of orientation-preserving isometries on $\C$ (see Section~\ref{sec:cryst-groups-latt}).

Indeed, let $G$ be the group of all 
maps $g\colon\C\to \C$ of the form
\begin{equation*}
  g(z) = \pm z + m + n\iu,
\end{equation*}
where  $m,n\in \Z$. Then  $G$ is  a 
 crystallographic group and the map  $\Theta$ is 
{\em induced}\index{map!induced by group action}\index{induced by group action} 
by $G$ in the sense that 
$\Theta(z) =\Theta(w)$ for  $z,w\in \C$ if and only if there exists $g\in G$ such that 
$w= g(z)$ (related concepts and facts are discussed in more detail in Section~\ref{sec:appquotmaps}).  This implies that 
the  quotient $\C/G$ can be identified with $\CDach$, and the map $\Theta$ with the quotient map $\C\ra \C/G$ (see 
Corollary~\ref{cor:groupquot}).   

The property that allows us to pass to a quotient $f$ in \eqref{eq:def_lattes2}
 is that the map  
 $z\mapsto A(z)=2z$ is   
\emph{$G$-equivariant} 
(see Lemma~\ref{lem:f_groupdescend}). This means 
 that $A$ maps  points that are in the same $G$-orbit to
points that are also in the same $G$-orbit, or equivalently that
\index{G-equivariant@$G$-equivariant}  
\index{equivariant}
\index{group action!map equivariant under}
\begin{equation}
  \label{eq:equivariant}
  A \circ g \circ A^{-1}\in G \text{ for all } g\in G. 
\end{equation}

The translations in $G$ form a subgroup $\Gtr$ 
 isomorphic (as a group)
 to $\Z^{2}\cong\Z\oplus \Z\iu$. 
 The
quotient $\C/\Gtr$ is naturally a {\em complex torus} $\T$, i.e., a Riemann
surface whose underlying $2$-manifold is a $2$-dimensional torus (for more on tori see Section~\ref{sec:applifttorend}). The
maps $A\colon \C\to \C$ and  $\Theta\colon \C \to \CDach$ descend to
$\T$, and we obtain  
 holomorphic maps $\overline{A} \colon \T \to \T$ and 
$\overline{\Theta} \colon \T \to \CDach$ such that $f\circ
\overline{\Theta} = \overline{\Theta} \circ \overline{A}$. So we have the  following commutative  diagram:
\begin{equation}
  \label{eq:def_lattes3}
  \xymatrix{
    \T\ar[r]^{\overline{A}} \ar[d]_{\overline{\Theta}} 
    & \T \ar[d]^{\overline{\Theta}}
    \\
    \CDach \ar[r]^{f} & \CDach\rlap{.}
  }
\end{equation}

We call a non-constant  holomorphic map $\overline A\: \T \ra \T$  on a complex torus $\T$ a {\em holomorphic torus endomorphism}. 
The Riemann-Hurwitz formula \eqref{eq:Riemann-Hurwitz}
implies that such a map $\overline{A}$ has no critical  points and is hence a 
 covering map (in the usual topological sense; see Section~\ref{sec:covmaps}).  
\index{torus!endomorphism!holomorphic}
\index{holomorphic torus endomorphism}

The relations  of our   example $f$ to crystallographic groups or to holomorphic torus endomorphisms hold for a more general class of rational maps, called {\em Latt\`es maps},
as the following statement shows.

\begin{theorem}[Characterization of Latt\`es maps]\label{thm:Lattesstruc}  Let $f\: \CDach \ra \CDach$ be
  a map. Then the following conditions are equivalent:  

\begin{enumerate}
 
 \item \label{item:Lattessrucii} $f$ is a rational Thurston map that
   has a parabolic orbifold and no  periodic critical points. 
  \index{parabolic!orbifold} 
  \index{orbifold!parabolic} 
  \index{Thurston map!parabolic}
  \index{Thurston map!rational}
    
 \item \label{item:Lattessruciii} There exists a crystallographic
   group $G$, 
   a $G$-equivariant  holomorphic  map $A\: \C\ra \C$ of the form 
   $A(z)=\alpha z+\beta$, where $\alpha, \beta\in \C$, $\abs{\alpha}>
   1$, and a holomorphic map $\Theta\: \C\ra \CDach$ induced by 
   $G$ such that $f\circ \Theta=\Theta\circ A$. 
   
 \item 
   \label{item:Lattesruciv}
   There exists  a complex torus $\T$, a holomorphic torus endomorphism
   $\overline{A} \colon \T\to \T$ with $\deg (\overline{A}) >1$, and a  non-constant 
   holomorphic map $\overline{\Theta} \colon \T \to \CDach$ such that
   $f\circ \overline{\Theta} = \overline{\Theta} \circ \overline{A}$. 
 \end{enumerate} 
\end{theorem}

So in \ref{item:Lattessruciii} the map $f$ is given as in \eqref{eq:def_lattes2}, and in 
\ref{item:Lattesruciv} as in \eqref{eq:def_lattes3}. We will see that we have $\deg(f)=\deg(\overline A)=|\alpha|^2>1$ (Lemma~\ref{lem:deglatttype}). 

As we already indicated, the previous theorem  motivates   the following definition. 
\begin{definition} [Latt\`es maps]\label{def:Lattes} 
  A map  $f\:\CDach \ra \CDach$ is called a 
  {\em Latt\`es map}\index{Latt\`{e}s map|textbf} 
  if it satisfies one of the  conditions (and hence every condition)  in  
  Theorem~\ref{thm:Lattesstruc}.  
\end{definition} 

The terminology is not uniform in the literature and some authors use the 
term ``Latt\`es map'' with a slightly different meaning (see the discussion in Section~\ref{sec:examples-lattes-maps}). Latt\`es maps became more widely known through Latt\`es paper \cite{La}, but they had been studied  about  half a century earlier by Schroeder, for example. See \cite{Mi} for more on the history of these maps.

Theorem~\ref{thm:Lattesstruc} is well known (see, for example, \cite{Mi}). We will
prove it  in Sections~\ref{sec:cryst-groups-latt} and~\ref{sec:quot-torus-endom}. The map $A$  in \ref{item:Lattessruciii} is
subject to strong further restrictions. See
Proposition~\ref{prop:alph_beta_G} (or \cite{Mi} and
\cite[Appendix]{DH84}) for more details.

  By Proposition~\ref{prop:rationalexpch}, condition
  \ref{item:Lattessrucii} 
 can equivalently be expressed as:
  \begin{enumerate}
  \item [(i')] \textit{$f$ is a rational  Thurston map that has a parabolic
    orbi\-fold and  is expanding. } 
  \end{enumerate}
  This is how we introduced Latt\`es maps in the beginning of the chapter. 
  
  The most convenient way to construct Latt\`es maps is based on 
  condition~\ref{item:Lattessruciii} in  Theorem~\ref{thm:Lattesstruc}. One starts with a crystallographic group $G$ not isomorphic to $\Z^2$ and an $G$-equivariant map $A$ as in this statement. Then there exists a holomorphic branched covering map $\Theta\: \C \ra \CDach$ induced by $G$; it is unique up to postcomposition 
  with a 
  M\"obius transformation (see Proposition~\ref{prop:G_yields_T}).  
  The existence of a Latt\`es map $f$ as in \eqref{eq:def_lattes2} then follows from the $G$-equivariance of $A$ (see Lemma~\ref{lem:f_groupdescend}).   
  
If  $\alpha_{f}\colon \CDach \to \widehat{\N}$ is the   ramification
  function  of $f$ (see Definition~\ref{def:weightf}), then   
  $\Theta \colon \C \to \CDach$  is a holomorphic branched covering map such that 
  $\deg(\Theta, z)= \alpha_{f}(\Theta(z))$ for all $z\in \C$ 
  (see Corollary~\ref{cor:orbofLattty}). Therefore, $\Theta$ is in fact
  the 
\emph{universal orbifold covering map}\index{universal orbifold covering map}\index{orbifold!universal covering map} 
of $\OC_{f}=(\CDach, \alpha_f)$ (see
  Theorem~\ref{thm:orbunifparacas} and Section~\ref{sec:orbifolds-coverings}).

It is quite natural to consider more general maps $f$ as in  \eqref{eq:def_lattes2} or 
as in \eqref{eq:def_lattes3},  where the  maps involved are
branched covering maps, but not necessarily holomorphic. To state this more precisely, we first  recall some terminology. 

 As usual, we call a map  $A\: \R^2\ra \R^2$ {\em affine}, if it has  the form 
\begin{equation}
  \label{eq:affine} 
  A(u)=L_A(u)+u_0, \quad u\in \R^2,
\end{equation} 
 where $L_A\: \R^2\ra \R^2$ is $\R$-linear and $u_0\in \R^2$. We call
 $L_A$ the {\em linear part} of $A$.

 Let $G$ be a crystallographic group  not isomorphic to
 $\Z^{2}$. Then one can show that the quotient $\R^2/G$ is homeomorphic to a $2$-sphere $S^2$ and the quotient map $\Theta\colon \R^{2} \to S^{2}\cong \R^2/G$ is  a branched covering map 
 induced by $G$. 
 %
 If, in addition, $A\colon \R^{2} \ra \R^{2}$ is an affine map
 that is $G$-equivariant and whose linear part $L_{A}$ satisfies 
 $\det (L_{A}) > 1$, then 
 there is a Thurston 
 map $f\colon S^{2}\to S^{2}$ such that the diagram
 \begin{equation}
   \label{eq:def_lattes_type}
   \xymatrix{
     \R^{2} \ar[r]^{A} \ar[d]_{\Theta} & \R^{2} \ar[d]^{\Theta}
     \\
     S^{2} \ar[r]^{f} & S^{2}
   }
 \end{equation}
 commutes (see the beginning of Section~\ref{sec:lattes-type-maps}). This is the basis of the following definition.

\begin{definition}[Latt\`es-type maps]
  \label{def:Lattestype}
  \index{Latt\`{e}s-type map} 
  Let $f\:S^2\ra S^2$ be a map
  such that there exists a crystallographic group $G$, an affine
  map $A\: \R^2\ra \R^2$ with $\det(L_A)>1$ that is
  $G$-equivariant, and a branched covering  map $\Theta\: \R^2\ra
  S^2$ induced by $G$ such that  $f\circ \Theta=\Theta\circ
  A$. Then $f$ is called a  {\em Latt\`es-type map}.
\end{definition} 
 So  Latt\`es-type maps are given 
  as in
  \eqref{eq:def_lattes_type}, where $A$ is affine. It  is clear that every Latt\`es map belongs to this class. For the  degree of a Latt\`es-type map $f$  we have  $\deg(f)=\det(L_A)$ 
(see Lemma~\ref{lem:deglatttype}). The requirement $\det(L_A)>1$ guarantees 
 the condition $\deg(f)\ge 2$ which is part of our definition of a Thurston map. One can also consider maps as in Definition~\ref{def:Lattestype} with  $\det(L_A)=1$ or 
 $\det(L_A)<0$. This gives   homeomorphisms and orientation-reversing maps, respectively.  According to our definition, no  such map  is  a Thurston map. 

Before we discuss some other properties of Latt\`es-type maps, we will first define  another   generalization of Latt\`{e}s maps based on
\eqref{eq:def_lattes3}. 
 We denote by
$T^{2}$ a $2$-dimensional torus, now considered as a purely 
topological object with no conformal structure. If  $\overline A\: T^{2} \ra T^{2}$ is  a branched
covering map, then again by the Riemann-Hurwitz formula \eqref{eq:Riemann-Hurwitz} the map
 $\overline A$ cannot have any critical points and 
must be an  orientation-preserving covering map. We call such a
map $\overline A$ a 
{\em (topological) torus endomorphism}.\index{torus!endomorphism} 
So a continuous map $\overline A\: T^{2}\ra T^{2}$ is
a torus endomorphism precisely if it is an orientation-preserving
local homeomorphism.

\begin{definition}[Quotients of  torus endomorphisms]
  \label{def:quotient_torus}
  Let $f\colon S^2\to S^2$ be a map on a $2$-sphere $S^2$
  such that there exists a torus endomorphism
  $\overline A\: T^2\ra T^2$ with $\deg(\overline A)\ge 2$, and a
  branched covering map $\overline \Theta\: T^2\ra S^2$ such that
  $f\circ \overline \Theta= \overline\Theta\circ \overline
  A$.
  Then $f$ is called a 
  {\em quotient of a torus endomorphism}.\index{quotient!of torus endomorphism}\index{torus!endomorphism!quotient of}
\end{definition}
In this case,  we have a commutative diagram of the form 
 \begin{equation}
  \label{eq:Lattesdef}
  \xymatrix{
    T^2 \ar[r]^{\overline A} \ar[d]_{\overline\Theta} &
    T^2 \ar[d]^{\overline\Theta}
    \\
    S^2 \ar[r]^f & S^2\rlap{.} 
  }
\end{equation}
Properties of quotients of torus endomorphisms are recorded in Lem\-ma~\ref{lem:simobserv}. In particular, every such map  is a Thurston map without
periodic critical points. 
Latt\`es-type maps are in this class.  

\begin{prop} \label{prop:immediate} Every Latt\`es-type map $f\:S^2\ra S^2$ is a quotient of a torus endomorphism and hence a Thurston map. It has a parabolic orbifold and no periodic critical points. 
\end{prop}

This implies that the orbifold of every  Latt\`{e}s-type map  has   one of the signatures
$(2,2,2,2)$, $(2,4,4)$, $(3,3,3)$, or $(2,3,6)$. The last three
signatures do not lead to genuinely new maps, as each Latt\`es-type
map whose orbifold has  such a signature is topologically conjugate to a Latt\`es map
 (Proposition~\ref{prop:rigidlattestype}). The
most interesting case is signature $(2,2,2,2)$. More details on
these  maps can be found in Example~\ref{ex:lattes_type} (see also 
Proposition~\ref{prop:2222} and Theorem~\ref{thm:Thurston_para}). These maps  include
\emph{flexible Latt\`{e}s maps} (see Definition~\ref{def:flex_Lattes}
and the discussion that follows there). 
The last statement in Proposition~\ref{prop:immediate} essentially  
characterizes Latt\`es-type maps among Thurston maps. 

\begin{prop}\label{prop:noperparaLTM}
Let $f\: S^2\ra S^2$ be a Thurston map.  Then $f$ is Thurston equivalent to a Latt\`es-type map if and only if $f$  has   a parabolic orbifold and no periodic critical points. 
\end{prop} 

If $f$ has a parabolic orbifold $\mathcal{O}_f$, but also  periodic critical points, then 
the signature of $\mathcal{O}_f$ is $(\infty, \infty)$ or $(2,\infty, \infty)$. 
It is easy to classify these maps up to Thurston equivalence as well (see Theorem~\ref{thm:para_Th_poly}).

Each Latt\`{e}s map is  expanding, but this is not always true for 
a  Latt\`{e}s-type map  (see
Example~\ref{ex:non-expanding-lattes}). One can state  a
simple criterion though  when this is  the case. Namely,  a Latt\`{e}s-type map
 is expanding 
if and only if the two (possibly complex) eigenvalues $\lambda_1$ and
$\lambda_2$ of the linear part $L_A$ of the affine map $A$ in \eqref{eq:def_lattes_type} satisfy
$|\lambda_1|, |\lambda_2|>1$ (Proposition~\ref{prop:expLattType}).

Every Latt\`{e}s map is a Latt\`{e}s-type map, and every
Latt\`{e}s-type map is a quotient of a torus endomorphism. 
On the other hand, not every Latt\`{e}s-type map is (conjugate to) a Latt\`{e}s map. It is a natural question whether every quotient of  
a torus endomorphism $f$ is Thurston equivalent
to a Latt\`{e}s-type map. One can show that this is true if $f$
is expanding (in this case, $f$ is even conjugate to a Latt\`{e}s-type map), but we have been unable to answer this question in full generality.  

Our presentation in this chapter is as follows. In 
Section~\ref{sec:cryst-groups-latt} we review 
crystallographic groups. We  also formulate two important existence and uniqueness statements for maps 
 related to crystallographic groups  or  parabolic orbifolds
(Proposition~\ref{prop:G_yields_T} and
Theorem~\ref{thm:orbunifparacas}), but  we postpone the proofs of
these facts to Section~\ref{sec:paraorbcov}. 
We then prove the implications \ref{item:Lattessruciii} $\Rightarrow$ \ref{item:Lattesruciv} and 
\ref{item:Lattessrucii} $\Rightarrow$ \ref{item:Lattessruciii} in
Theorem~\ref{thm:Lattesstruc}.
The final implication 
\ref{item:Lattesruciv} $\Rightarrow$ \ref{item:Lattessrucii}
is established  in Section~\ref{sec:quot-torus-endom}  after we discussed  some relevant  facts  about   quotients of torus endomorphisms.
 
In Section~\ref{sec:clas-latt-maps} we analyze  the restrictions on $\alpha$ and $\beta$ for the map 
$A(z)=\alpha z+\beta$ in Theorem~\ref{thm:Lattesstruc} in  detail. 
This is mostly for potential  future reference and can be omitted at first reading. Section~\ref{sec:lattes-type-maps} is devoted to 
 Latt\`es-type maps and their properties.  Here  we justify Proposition~\ref{prop:immediate} and 
 Proposition~\ref{prop:noperparaLTM}. The proof of this last statement is rather involved and uses some facts about mapping class groups that we will only cite from the literature, but not discuss in detail. As we will not use 
 Proposition~\ref{prop:noperparaLTM}  later, its  proof can safely be skipped.  
 
 We revisit  
crystallographic groups and parabolic orbifolds in Section~\ref{sec:paraorbcov}. Here we give proofs of  Proposition~\ref{prop:G_yields_T} and Theorem~\ref{thm:orbunifparacas}. We will  emphasize a geometric point of view. This will  help us in the discussion of some explicit Latt\`es
maps in Section~\ref{sec:examples-lattes-maps}.

\section{Crystallographic groups and Latt\`{e}s maps}
\label{sec:cryst-groups-latt}

In this section we  focus on maps as  in  statement \ref{item:Lattessruciii} of Theorem~\ref{thm:Lattesstruc}. We first review some facts  related to crystallographic groups. For a more detailed discussion  related to group actions and quotient spaces see
Section~\ref{sec:appquotmaps}. 

We use the notation 
$$\Aut(\C) =\{ z\in \C\mapsto \alpha z +\beta : \alpha,\beta\in
\C,\,  \alpha\neq 0\}$$
for\index{Aut(C)@$\Aut(\C)$} the group of all holomorphic automorphisms  of $\C$  and 
$$\Isom(\C) = \{z\in \C\mapsto \alpha z +\beta : \alpha,\beta\in
\C,\,  \abs{\alpha}=1\}\subset \Aut(\C)$$
for\index{Isom(C)@$\Isom(\C)$} the group of all
orientation-preserving isometries of $\C$ (equipped with the Euclidean
metric).

Let $G$ be a group of homeomorphisms acting on $\C$.  
If $z\in \C$, then we denote by $G_z\coloneqq \{g\in G: g(z)=z\}$
its {\em stabilizer subgroup}\index{stabilizer} 
and by $Gz\coloneqq \{g(z): g\in G\}$ its {\em orbit under $G$} or {\em $G$-orbit}.\index{orbit} 
The group $G$ induces a  natural equivalence relation on $\C$ whose 
equivalence classes are given by the $G$-orbits. 
The corresponding quotient space is denoted by $\C/G$. 
 
The group $G$ acts 
\emph{properly discontinuously}\index{properly discontinuous}\index{group action!properly discontinuous}
on $\C$ if  for each compact set $K\subset \C$ there are only
finitely many maps $g\in G$ with  $g(K) \cap K\neq \emptyset$. 
Then the stabilizer $G_z$ is finite for each $z\in \C$. 
 The group $G$ acts 
\emph{cocompactly}\index{group action!cocompact}\index{cocompact}
on $\C$ if there exists a compact set $K\sub \C$ such that the sets $g(K)$, $g\in G$, cover $\C$.  In this case, $\C/G$ is compact. 

We call  $G$  a 
{\em (planar) crystallographic group}\index{crystallographic group} if each
element $g\in G$ is an orientation-preserving isometry on 
$\C$  and if the
action of $G$ on $\C$ is properly discontinuous and cocompact. Note that this definition of a crystallographic group is more restrictive than usual, since we require that 
the isometries in $G$ are  orientation-preserving.

 We say that two crystallographic groups $G$ and $\widetilde G$ are {\em conjugate} if there exists $h\in \Aut(\C)$ such that 
$$\widetilde G= h\circ G \circ h^{-1}\coloneqq \{h\circ g\circ h^{-1}: g\in G\}. $$ 
The  following statement gives a
classification of crystallographic groups up to conjugacy.  

\begin{theorem}[Classification of crystallographic groups]
  \label{thm:G_signature}
  \index{crystallographic group}
  \index{planar crystallographic group}
  Let  $G\subset\Isom(\C)$ be a planar crystallographic 
  group. Then  $G$
  is conjugate  to one of the following groups
  $\widetilde{G}$ consisting of all $g\in \Isom(\C)$ of the form
  
   \begin{enumerate}
   \item[\upshape(torus)]
    \label{item:Gtorus} 
       $\quad \quad \quad\displaystyle z\mapsto g(z) = z + m +
    n\tau,\quad m, n\in \Z$; 

  \item[\upshape(2222)]
    \label{item:G2222}
     $\quad \quad \quad z\mapsto g(z) = \pm z + m + n\tau, \quad m,
    n\in \Z$;  
    
  \item[\upshape(244)]
    \label{item:G244}
     $\quad \quad \quad z\mapsto g(z)= \iu^kz+ m +n\iu,\quad m, n\in
    \Z, \,  k=0,1,2,3$; 
    
  \item[\upshape(333)]
    \label{item:G333}
    $\quad \quad \quad z\mapsto g(z)= \omega^{2k}z+ m + n\omega,\quad
    m,n\in \Z, \,  k=0,1,2$;  

  \item[\upshape(236)]
    \label{item:G236}
     $\quad \quad \quad z\mapsto g(z)= \omega^{k}z+ m +n\omega, \quad
    m,n\in \Z, \, k=0,\dots, 5$.
  \end{enumerate}
  
  Here $\tau\in \C$ is a fixed
   number with  $\imag(\tau)>0$ in the first two cases and $\omega= e^{\iu \pi/3}$ in the last two cases. 
\end{theorem}

This classification of (planar) crystallographic groups is
classical.  Proofs can be found in \cite{Be} and \cite{Ar}; see also
\cite{We}.  

We used Conway's orbifold notation for planar crystallographic
groups, see \cite{Con} (with one difference: Conway denotes the
(torus) case by $(\circ)$). In Theorem~\ref{thm:G_signature}
the group $G$ determines the type of the conjugate group
$\widetilde G$ uniquely.  Accordingly, we speak of a
crystallographic group $G$ of type $(2222)$, etc., if
$\widetilde G$ belongs to the corresponding class.  This
terminology is explained by the fact that the quotient space
$\C/G$ is a torus in the first case of the theorem. In the other cases,  
$\C/G$ is homeomorphic to $\CDach$, and  $G$ induces a natural ramification function
$\alpha$ on $\CDach=\C/G$ so that the orbifold
$(\CDach, \alpha)$ has a signature as indicated by the  type of $G$ (see
the discussion below and 
Section~\ref{sec:paraorbcov} for more details).

  \begin{figure}[p]
  \centering
  \begin{overpic}
    [width=7.1cm, tics=20,
    ]
    {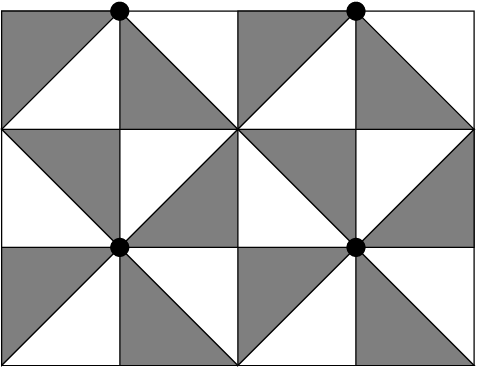}
  \end{overpic}
  \caption{Invariant tiling for type  $(244)$.} 
  \label{fig:244_univ_orb}
\end{figure}

\begin{figure}[p]
  \centering
  \begin{overpic}
    [width=7.1cm, tics=20,
    ]
    {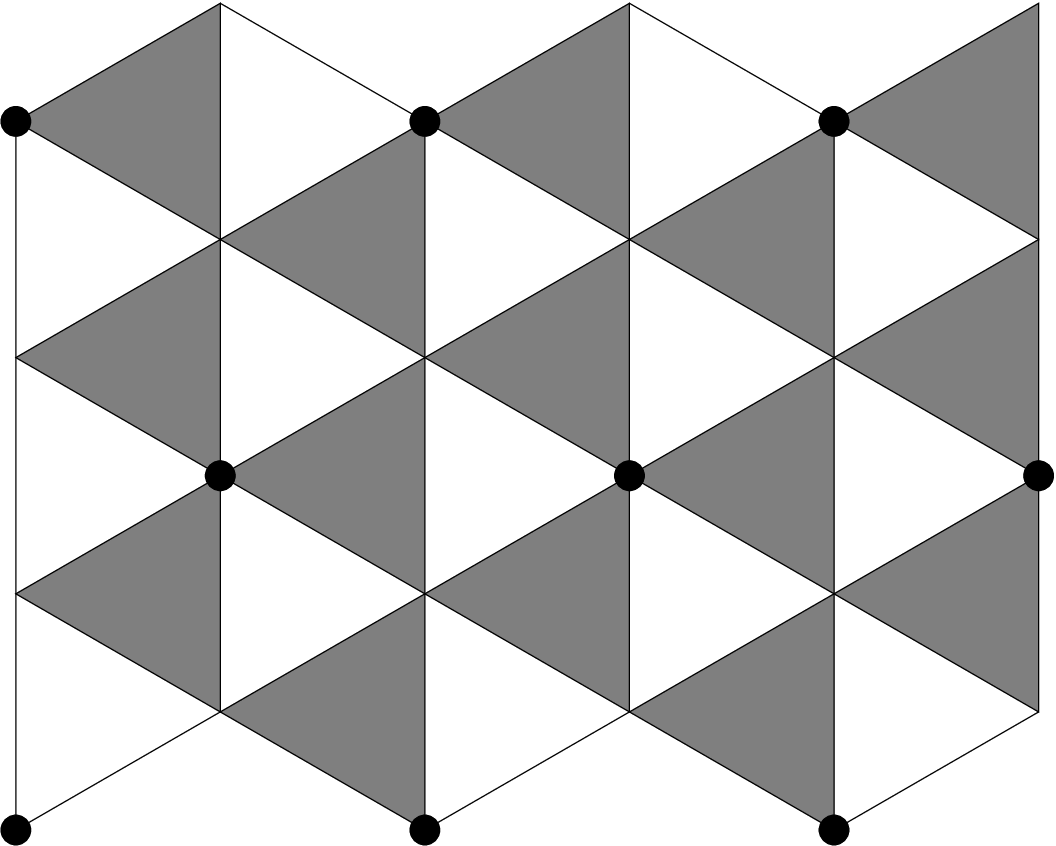}
  \end{overpic}
  \caption{Invariant tiling for type $(333)$.} 
  \label{fig:333_univ_orb}
\end{figure}

\begin{figure}[p]
  \centering
  \begin{overpic}
    [width=7.1cm, tics=20,
    ]
    {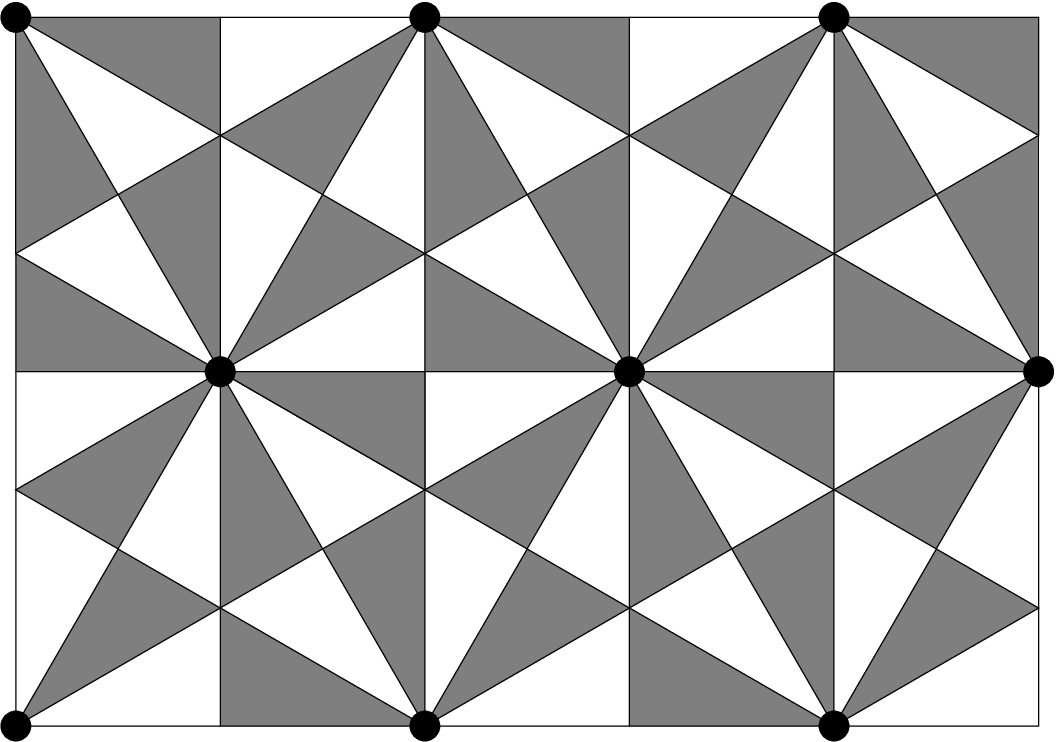}
  \end{overpic}
  \caption{Invariant tiling for type $(236)$.} 
  \label{fig:236_univ_orb}
\end{figure}

\begin{rem}
  \label{rem:crys_semi}   A crystallographic group $G$ is isomorphic as a group  to its conjugate 
  $\widetilde{G}$. This implies that  if $G$ is of (torus) type, then $G$ is isomorphic  to  $\Z^2$. 
 In the other cases,  $G$ is
  isomorphic to a semidirect product $\Z^2\rtimes \Z_k$ of $\Z^2$ and a cyclic group 
  $\Z_k=\Z/k\Z$. Here $k=2,4,3,6$ if ${G}$ is of type
  $(2222)$, $(244)$, $(333)$, or $(236)$,  respectively. To see this, one considers 
  the isomorphic group $\widetilde G$ and identifies   $\Z^2$ with the lattice $\Gamma=\Z\oplus \Z\tau$, where
  $\tau\in \C$ satisfies $\imag (\tau)>0$ in  case $(2222)$,
  $\tau=\iu$ in  case $(244)$, and $\tau= \omega$ in 
  cases $(333)$ and $(236)$. In addition, we identify $\Z_k$ with
  the multiplicative group consisting of the  $k$-th  roots of unity. Then  each element in 
  $\Z_k$ acts by multiplication as an automorphism on $\Gamma\cong \Z^2$ and one derives the isomorphism
   $G \cong \widetilde G \cong \Z^2\rtimes \Z_k$. 
 All this  is standard and well known, and  we skip the details. 
\end{rem}

If $G$ is a crystallographic group, we denote by 
${G}_{\text{tr}}$\index{Gaa@${G}_{\text{tr}}$}
the subgroup consisting of all translations in $G$, i.e., ${G}_{\text{tr}}$ consists of all maps $g\in G$ of the form $z\mapsto g(z)=z+\ga$ with $\gamma\in \C$.  Theorem~\ref{thm:G_signature} implies that 
${G}_{\text{tr}}$ is a normal subgroup of finite index in $G$, and that 
there is a  lattice $\Gamma\sub \C$ such that 
${G}_{\text{tr}}=\{z\mapsto z+\ga: \ga\in \Gamma\}$. This lattice $\Gamma$ has {\em rank}
$2$ in the sense that it spans $\R^2\cong\C$ (see Section~\ref{sec:applifttorend} for more discussion).  We call it 
the {\em underlying lattice} of the crystallographic group $G$.  
For the group $\widetilde{G}$ as in Theorem~\ref{thm:G_signature} it 
 is equal to $\Z\oplus  \Z \tau $ in  cases
(torus) and $(2222)$,
to $\Z \oplus \Z   \iu $ 
in case
$(244)$,
and to $\Z\oplus  \Z  \omega $ in cases
$(333)$ and $(236)$.
A crystallographic group $G$ is of  (torus) type if and only if $G={G}_{\text{tr}}$, or, equivalently, if and only if $G$ is isomorphic to  $\Z^2$.

Crystallographic groups  ${G}$  of type
$(244)$, $(333)$, and $(236)$
 are represented in   Figure~\ref{fig:244_univ_orb},
Figure~\ref{fig:333_univ_orb}, and
Figure~\ref{fig:236_univ_orb}, respectively. 
For each type   $(abc)$ we see (part of) a tiling of $\C$ given by  isometric copies of Euclidean triangles with angles $\pi/a,\pi/b, \pi/c$.  
The corresponding group
${G}$ consists of all orientation-preserving isometries of the plane $\C$ that keep the
tiling invariant. This means that each group element $g\in {G}$ maps
each triangle in the tiling to another one of 
the same color. Moreover,  $g$ maps each point marked by a black dot to
another such point. If we assume that $0\in \C$ is one of these points, 
then  all points marked by a black dot  form the 
underlying lattice $\Gamma$ of ${G}$, i.e.,
the orbit of $0\in \C$ by the group of translations
${G}_{\text{tr}}\subset {G}$. 
Actually, the full 
${G}$-orbit of $0$ 
  is  then equal to  its
${G}_{\text{tr}}$-orbit $\Gamma$.

Suppose $G$ is a crystallographic group and   $z,w\in\C$ are  contained in the same $G$-orbit. Then  
 their stabilizers $G_z$ and $G_w$ have the same order
$\#{G_z}=\#{G_w}$, since they are conjugate subgroups of $G$. 
Each  non-trivial stabilizer $G_z$ is cyclic and of order $2$, $3$, $4$, or $6$,  as can be seen from Theorem~\ref{thm:G_signature} (actually, an independent proof of this  fact 
is one of the main steps in the  proof of Theorem~\ref{thm:G_signature}). The numbers in the label for the type of a group indicate the orders of non-trivial stabilizers
in  distinct orbits. For example,   a  crystallographic group of type $(236)$ has
three distinguished orbits consisting of points with non-trivial stabilizers
of order $2$, $3$, or $6$.  Note that every $g\in G_z$ is a rotation around $z$, since it is
an orientation-preserving isometry that fixes $z$.

Crystallographic groups are closely related to parabolic
orbifolds $(\CDach, \alpha)$ with a finite ramification function 
$\alpha\: \CDach\ra \N$ (so  $\alpha$ satisfies  $\alpha(p)<\infty$ for
$p\in \CDach$). These orbifolds have one of the signatures
$(2,2,2,2)$, $(2,4,4)$, $(3,3,3)$, or $(2,3,6)$.  This
corresponds precisely to the types of crystallographic groups
that are not of (torus) type, i.e., not isomorphic to $\Z^2$. We will  formulate two related   statements for immediate use and easy reference (Proposition~\ref{prop:G_yields_T} and Theorem~~\ref{thm:orbunifparacas}), but discuss 
 their proofs only later  in   Section~\ref{sec:paraorbcov} (throughout we also rely on material discussed in the appendix). 

We consider a crystallographic group
$G$ not isomorphic to $\Z^2$. Let
$\C/G$ be the quotient space. 
The quotient map  $\Theta_G\: 
\C\ra \C/G$ sends a point $z\in \C$ to 
its $G$-orbit $Gz$, considered as an element of $\C/G$ (see Section~\ref{sec:appquotmaps} for some general facts related to this). The quotient $\C/G$ is a topological $2$-sphere (see the discussion below). Moreover, one can equip 
$\C/G$ with natural geometric and conformal structures
so that 
the map  $\Theta_G\: \C \ra \C/G$  is holomorphic and one has a parabolic orbifold associated with $G$ (see  Section~\ref{sec:paraorbcov}). 

In our context it is convenient to allow   a more flexible setup, where we are not tied to the quotient space and the quotient map.   
So let $\Theta\colon \C\to S^2$ be  a continuous map into   a topological $2$-sphere $S^2$. We say that
 $\Theta$ is \emph{induced} by $G$ if it has the 
 following property: $\Theta(z)=\Theta(w)$ for 
 $z,w\in \C$ if and only if there exists $g\in G$ with $w=g(z)$.
The quotient map $\Theta=\Theta_G$ satisfies  this condition. We will see momentarily that  there is a close relation between the quotient map $\Theta_G$  and an arbitrary map $\Theta$ induced by $G$.
 The relevant facts  are summarized in the following statement. 
 
\begin{prop}
  \label{prop:G_yields_T}
  \index{map!induced by group action}
  \index{induced by group action}
  Let $G$ be a crystallographic group not isomorphic to $\Z^2$.  Then there exists a
  holomorphic branched covering map $\Theta\colon \C \to \CDach$ that is induced by
  $G$.  Associated with $\Theta$  is a  unique finite ramification function
  $\alpha\: \Cdach\ra \N$ such that  
    \begin{equation}
      \label{eq:degT_orderG}
     \alpha(\Theta(z))= \deg(\Theta,z) = \#G_z \text{ for $z\in \C$.}
  \end{equation}
The orbifold $(\Cdach,\alpha)$ is parabolic.  
  
  If $\widetilde \Theta\: \C \ra S^2$ is another continuous map induced by $G$, then there exists a unique homeomorphism 
  $\varphi\: \CDach \ra S^2$ such that $\widetilde \Theta= 
  \varphi \circ 
  \Theta$. If here  $S^2=\CDach$ and $\widetilde \Theta$ is holomorphic, then  $\varphi$ is a M\"{o}bius
  transformation.  
\end{prop}

For the notion of a branched covering map that applies here  see Definition~\ref{def:brcovmap}.
In  Section~\ref{sec:paraorbcov} we will present  an  explicit, somewhat lengthy, geometric construction of $\Theta$.

%
%

A consequence of Proposition~\ref{prop:G_yields_T} is   that every continuous map $\widetilde \Theta\: \C \ra S^2$ induced by $G$ is a branched covering map. In particular,  $ \widetilde \Theta$ is surjective. 
If we combine this with the fact that  $\widetilde \Theta$ is  induced by $G$, then we  can easily see that the map $Gz \in \C/G \mapsto \widetilde \Theta(z)\in S^2$ is a bijection between $\C/G$ and $S^2$.
 Corollary~\ref{cor:groupquot}~\ref{item:groupquot1}   implies that  this  map is actually 
 a homeomorphism between these spaces.  
It allows us 
to identify  $\C/G$ and  $S^2$.  Under this identification,   
  $\widetilde\Theta$   corresponds to the quotient map 
  $\Theta_G$. This also shows that  if a crystallographic group $G$ is not isomorphic
  to $\Z^2$, then the quotient space $\C/G$ is indeed a $2$-sphere. In
  Section~\ref{sec:paraorbcov} we will provide a more  explicit  geometric argument to justify this fact. If $G$ is isomorphic to $\Z^2$, then $\C/G$ is clearly a  ($2$-dimensional) torus.   
  
 Relation  \eqref{eq:degT_orderG} together with  the fact that 
  $Gz\in \C/G\mapsto \Theta(z)\in \CDach$ is a 
bijection implies that the orbifold 
$(\CDach, \alpha)$ in Proposition~\ref{prop:G_yields_T} has a signature that corresponds to  the type of $G$.    
 So if $G$ has type $(244)$, for example, then 
 $(\CDach, \alpha)$ has signature $(2,4,4)$.

Instead of starting with a crystallographic group and obtaining  an associated parabolic orbifold as in Proposition~\ref{prop:G_yields_T}, one can also reverse this process. 
This is based on the existence of the universal orbifold covering  map which is discussed in detail in the appendix (see
Section~\ref{sec:orbifolds-coverings}).  For the present purpose
we will formulate a relevant special case explicitly.  First, we
recall some terminology.

If  $\Theta\: \C \ra \CDach$  is a   branched covering map, then 
a 
{\em deck transformation}\index{deck transformation} 
of $\Theta$ is a homeomorphism $g\: \C \ra \C$ 
such that $\Theta\circ g=\Theta$. If $\Theta$ is holomorphic, then 
this is also true for each   deck transformation $g$ of $\Theta$ (see the last part of Lemma~\ref{lem:2_3_branched}) and  so $g\in \Aut(\C)$. 
The  deck transformations of $\Theta$ form a group $G$.  

We say that  $\Theta\: \C \ra \CDach$ is a 
{\em regular} branched covering map\index{branched covering map!regular}\index{regular branched covering map}  
if its deck transformations act transitively on the fibers of  $\Theta$; this means that if $z,w\in \C$ and $\Theta(z)=\Theta(w)$, then there exists $g\in G$ such that $w=g(z)$. Note that $\Theta$ is regular if and only if it is induced by its group of deck transformations $G$.

\begin{theorem}\label{thm:orbunifparacas} Let $(\CDach, \alpha)$ be a parabolic orbifold with a finite ramification function 
$\alpha\: \CDach\ra \N$. Then there exists 
a   holomorphic branched covering map $\Theta\: \C \ra \CDach$ such 
that \begin{equation}
 \label{item:paraorb1}
 \text{ $\deg(\Theta,z)=\alpha(\Theta(z))$ for $z\in \C$. }
 \end{equation} 
 The branched covering map $\Theta$ is regular and its deck transformation group $G$ is a crystallographic group with  
 \begin{equation}
      \label{eq:degT_orderG2}
     \alpha(\Theta(z))= \deg(\Theta,z) = \#G_z \text{ for  $z\in \C$.}
  \end{equation}
 
 Moreover,  if  $\widetilde \Theta\:  \C \ra \CDach$ is another holomorphic map satisfying 
 \eqref{item:paraorb1},  then there exists $h\in \Aut(\C)$ such that 
$\widetilde \Theta=\Theta\circ h$.  
\end{theorem} 

The essentially unique map $\Theta$  is the {\em universal orbifold covering map} of the orbifold  $(\CDach, \alpha)$ 
(see Section~\ref{sec:orbifolds-coverings}). 

 We will  present the proof for  the existence  of $\Theta$ in  Section~\ref{sec:paraorbcov}. We also refer the reader to the appendix, where more general facts  are discussed from which 
 Theorem~\ref{thm:orbunifparacas} can be derived. More specifically, the existence of  $\Theta$ 
 is a special case of Theorem~\ref{thm:univ_orbi_cover}. The uniqueness statement for $\Theta$ follows from Corollary~\ref{cor:unique_univ_orbi}
and Remark~\ref{rem:holoorbset}. Finally, the statement about the deck transformation group of $\Theta$ follows from 
Proposition~\ref{prop:prop_deck_trafo}.

Since  $\Theta\: \C \ra \CDach$  is regular, it is  induced by the crystallographic 
group $G$ given by its deck transformations.
 In particular, $\C/G$ is a topological $2$-sphere
 (this follows from Corollary~\ref{cor:groupquot}~\ref{item:groupquot1})
  and so $G$ is not  isomorphic to $\Z^2$ (in which case $\C/G$ is a torus).  
  
  The relation   \eqref{eq:degT_orderG2}  again implies that the crystallographic group $G$ arising 
in Theorem~\ref{thm:orbunifparacas} has a type corresponding to  the signature of the orbifold $(\CDach, \alpha)$.

If $\Theta$ is as in Proposition~\ref{prop:G_yields_T},  it is obviously  the universal orbifold covering map of its associated 
orbifold $(\CDach, \alpha)$. So in a sense, this proposition and  Theorem~~\ref{thm:orbunifparacas} tell the same story from different perspectives. In Proposition~\ref{prop:G_yields_T}
the crystallographic group $G$ is given, and from the map $\Theta$ in this proposition we obtain a
ramification function $\alpha$  on $\CDach$ whose associated  orbifold  has  a unique signature corresponding to the type of $G$.  Here $\Theta$ 
and    $\alpha$ are  only determined up to post- or  precomposition with a M\"obius transformation, respectively.  In  
Theorem~\ref{thm:orbunifparacas} the ramification function $\alpha$  is fixed, while $G$ is only unique up to conjugation by an element in $\Aut(\C)$.

As we will see in Section~\ref{sec:paraorbcov}, one can use this relation between these statements and  reduce  the existence 
proof of  the orbifold covering map 
 $\Theta$  in 
 Theorem~\ref{thm:orbunifparacas} to Proposition~\ref{prop:G_yields_T} by a suitable choice of the crystallographic group $G$.

 We are now ready to prove one of the implications of
 Theorem~\ref{thm:Lattesstruc}.

\begin{proof}[Proof of \ref{item:Lattessruciii} 
  $\Rightarrow$ \ref{item:Lattesruciv} in 
  Theorem~\ref{thm:Lattesstruc}]
  \label{proof_ii_iii} Assume $f\: \CDach \ra \CDach$ is  a map as in  \ref{item:Lattessruciii}.  Then there exists a 
  crystallographic group $G$, a $G$-equivariant map $A\colon\C\to
  \C$ of the form $A(z)=\alpha z+\beta$ where $\alpha,\beta\in
  \C$, $|\alpha|>1$, and  a holomorphic map $\Theta\colon \C\to \CDach$  induced by
  $G$ such that $f\circ \Theta=
  \Theta\circ A$.
   Note that then  $\C/G$ is homeomorphic to 
  $\CDach$ and so $G$ is not isomorphic to $\Z^2$.  It follows   from Proposition~\ref{prop:G_yields_T} that $\Theta$ is a branched covering map.

  Let $\Gtr\sub G$ be the normal subgroup of translations in $G$.
  Consider the quotient space $\T\coloneqq \C/\Gtr$ and the
  quotient map $\pi \colon \C\to \T=\C/\Gtr$. Then $\T$ is a
  topological torus and $\pi$ is a covering map. Moreover, there
  is a natural conformal structure on $\T$ so that $\T$ is a
  complex torus and $\pi$ is holomorphic (see
  Section~\ref{sec:applifttorend} for more details).

Let  $g\in \Gtr$ with $g\ne \id_{\C}$ be arbitrary.
Since  $A$ is $G$-equivariant, we have $\widetilde g\coloneqq A\circ g \circ A^{-1}\in G$. So the map 
$\widetilde g$ is an orientation-preserving isometry on $\C$. Moreover, this map has no fixed points, because this is true for its  conjugate map $g$, which is a translation.  
Hence 
$\widetilde g\in G$  is also a translation, i.e., $\widetilde g\in  \Gtr$. 
This shows that $A\circ g \circ A^{-1}\in \Gtr$ whenever $g\in \Gtr$. 

We conclude that $A$ is $\Gtr$-equivariant and so  descends to the quotient $\T=\C/\Gtr$
(see Lemma~\ref{lem:f_groupdescend}). More explicitly, if we set 
\begin{equation*}
  \overline{A}(\pi(z)) \coloneqq  \pi(A(z))
\end{equation*}
for  $z\in \C$, then $\overline{A} \: \T \ra \T$ is a well-defined non-constant continuous map 
 satisfying 
$\overline{A}\circ \pi = \pi \circ A$.  Since $\pi$ and $\pi\circ A$ are holomorphic, the map $\overline{A}$ is  holomorphic as well, because locally  $\overline{A}$ can be written as $\pi \circ A \circ  \pi^{-1}$ for a suitable (holomorphic) branch of  
$\pi^{-1}$ (alternatively, one can apply 
Lemma~\ref{lem:2_3_branched}). It follows that  $\overline A$ is a holomorphic torus endomorphism.

Similarly, the map $\Theta$  descends to $\T$. Indeed, suppose 
$z,w\in \C$ and $\pi(z)=\pi(w)$. Then there exists $g\in \Gtr$ such that $w=g(z)$. Since $\Theta$ is induced by $G\supset 
\Gtr$, we then have $\Theta(z)=\Theta(w)$. So if we set  
\begin{equation*}
  \overline{\Theta}(\pi(z)) \coloneqq  \Theta(z)
\end{equation*}
for  $z\in \C$, then we get a well-defined map  $\overline{\Theta}\colon \T\to \CDach$   such that    $ \overline{\Theta}\circ \pi= \Theta$. This last relation implies that $\overline \Theta$ is  non-constant  and holomorphic, and a branched covering map (see 
Lemma~\ref{lem:2_3_branched}~\ref{item:2_out_3_2}). 

Note that 
\begin{equation*}
  f \circ \overline \Theta \circ \pi
  = 
  f \circ  \Theta
  =
  \Theta\circ A 
  = 
  \overline{\Theta}\circ \pi \circ A
  =
  \overline{\Theta}\circ \overline{A} \circ \pi.
\end{equation*}
Since $\pi\: \C \ra \T $ is surjective, it follows  that  $f \circ \overline \Theta
=\overline \Theta \circ \overline{A}$.

The holomorphic maps considered, and their relations, can be summarized in the following commutative diagram: 
\begin{equation}
  \label{eq:holom_qu_diagram}
  \xymatrix{
    \C \ar[r]^{A} \ar[d]_{\pi} \ar@/_2pc/[dd]_{\Theta} 
    & \C \ar[d]^{\pi}\ar@/^2pc/[dd]^{\Theta}
    \\
    \T \ar[r]^{\overline{A}}\ar[d]_{\overline{\Theta}} 
    & \T\ar[d]^{\overline{\Theta}}
    \\
    \CDach \ar[r]^{f} & \CDach\rlap{.}
  }
\end{equation}
This shows that $f$ is a quotient of a holomorphic torus endomorphism. 
If we have maps as in  \eqref{eq:holom_qu_diagram}, then  
\begin{equation}
  \label{eq:degf_a2}
  \deg(f)=\deg(\overline{A})=\abs{\alpha}^2
\end{equation}
(recall that $A(z)=\alpha z +\beta$). 
We postpone the justification of this to  Lemma~\ref{lem:deglatttype}, 
  where we will establish  a more general fact (not relying on  Theorem~\ref{thm:Lattesstruc}, of course). 
In particular, $\deg(\overline{A})=\abs{\alpha}^2>1$, and so $\deg(\overline{A})\ge 2$. 
It follows that $f$ is indeed a map as in \ref{item:Lattesruciv}. 
\end{proof}

\begin{rem}
  The action of the crystallographic group $G$ descends to
the complex torus  $\T=\C/\Gtr$;  more precisely, the  group 
 $\overline{G}=G/\Gtr$ acts naturally on
  $\T$. It easily follows from Theorem~\ref{thm:G_signature} that $\overline{G}$ is  a cyclic group. 
   Accordingly, condition \ref{item:Lattesruciv} in 
  Theorem~\ref{thm:Lattesstruc} 
can  be formulated similarly 
to
  condition~\ref{item:Lattessruciii}  in terms of a  cyclic group action 
   on $\T$. This is discussed in \cite{Mi} in more detail. We  do not  pursue this point of view here, since the underlying  geometry
   is not as easy  to visualize as  for crystallographic groups. 
\end{rem}

\begin{proof}[Proof of \ref{item:Lattessrucii} 
  $\Rightarrow$ \ref{item:Lattessruciii} in 
  Theorem~\ref{thm:Lattesstruc}] Let $f\: \CDach\ra \CDach$  be a map as in \ref{item:Lattessrucii}, i.e., a rational  Thurston map with  a parabolic orbifold  $\mathcal{O}_f=(\CDach, \alpha_f)$ and no periodic critical points. By Proposition~\ref{prop:otherramprops}~\ref{item:rami_infty}
   we then have $\alpha_f(p)<\infty$ for  $p\in \CDach$. 
   
   Let $\Theta\colon \C\to \CDach$ be
  the (holomorphic) universal orbifold covering map of the orbifold $\mathcal{O}_f$ (see
  Theorem~\ref{thm:orbunifparacas}) and  $G$ be the group of deck
  transformations of $\Theta$. Then   $G$ is a crystallographic group and $\Theta$ is induced 
  by $G$.

We have to find an automorphism  $A\colon \C\ra \C $ such that
$f\circ\Theta= \Theta\circ A$. For this we consider the  holomorphic map
$f\circ \Theta\colon \C \ra \CDach$. Since $f\: \Cdach \ra \Cdach$  and 
$\Theta\: \C \ra \Cdach$ are branched covering maps, $f\circ \Theta\: \C \ra \Cdach$ is a branched covering map as well (see Lemma~\ref{lem:2_3_branched}~\ref{item:2_out3_1}). 

Since $\mathcal{O}_f=(\CDach, \alpha_f)$ is parabolic, by
Proposition~\ref{prop:parabolicOf} we have  
\begin{equation*}
  \label{eq:degparacond}
  \deg(f,p)\cdot \alpha_f(p)=
  \alpha_f(f(p))
\end{equation*} for  all $p\in \CDach$. For  each   $z\in \C $ we have  $\deg(\Theta, z)=\alpha_f(\Theta(z))$, and so 
\begin{align*}
  \deg(f\circ\Theta,z)& = 
  \deg(f,\Theta(z))\cdot \deg(\Theta,z) \\
  &=\deg(f,\Theta(z))\cdot 
 \alpha_f(\Theta(z)) =\alpha_f\big(f(\Theta(z))\big).
\end{align*}

This shows that  $f\circ \Theta$ is another
universal orbifold covering map of $\OC_f$. The essential uniqueness   of the universal
orbifold cover (see Theorem~\ref{thm:orbunifparacas}) implies that
there is an automorphism  $A\colon\C \ra \C$ satisfying 
$f\circ \Theta = \Theta \circ A$. In other words, we have a commutative diagram as in \eqref{eq:def_lattes2}. Since $\Theta$ is induced by $G$, it follows that the map $A$ is $G$-equivariant (see Lemma~\ref{lem:f_groupdescend}).

  The map  $A\in \Aut(\C)$ is necessarily of the form $z\mapsto A(z)=\alpha z + \beta$, where $\alpha,\beta\in \C$, $\alpha\ne 0$. 
  As in \eqref{eq:degf_a2}, we have 
  $\deg(f)= \abs{\alpha}^2\geq 2$, and so  $\abs{\alpha}>1$.
  It follows that $f$ is as in  \ref{item:Lattessruciii}. 
\end{proof}

\section{Quotients of torus endomorphisms and parabolicity} 
\label{sec:quot-torus-endom}
\index{quotient!of torus endomorphism}
\index{torus!endomorphism!quotient of}

We now prepare the proof of the implication
\ref{item:Lattesruciv} $\Rightarrow$~\ref{item:Lattessrucii} in
Theorem~\ref{thm:Lattesstruc}. As we will see, a map $f$ as in
Theorem~\ref{thm:Lattesstruc}~\ref{item:Lattesruciv} is indeed a
Thurston map. The main difficulty  is to show that $f$ has a
parabolic orbifold.
To address this, we will first establish 
  some general statements for  quotients of
endomorphisms on a torus $T^2$ (see  Definition~\ref{def:quotient_torus}).

\begin{lemma}[Properties of quotients of torus endomorphisms]
  \label{lem:simobserv} 
  Let $f\: S^2\ra S^2$ be a quotient
  of a torus endomorphism, and $\overline \Theta\: T^2\ra S^2$ and 
  $\overline{A}\colon T^2\to T^2$ with $\deg(\overline A)\ge 2$ be as in
 Definition~\ref{def:quotient_torus}. Then the following statements are true:

\begin{enumerate}

 \item \label{item:simobi}
 The map $f$ is a Thurston map without periodic critical points,
 and it satisfies $\deg(f)=\deg(\overline A)\ge 2$. 
 
 \item \label{item:simobii}
   The set $\post(f)$ is equal to the set of critical values of
   $\overline \Theta$, i.e., 
   \begin{equation*}
     \post(f) = \overline{\Theta}(\crit(\overline{\Theta})).     
   \end{equation*}

 \item 
   \label{item:simobiii}
   The ramification function of $f$ is given by 
\begin{equation*} 
\alpha_f(p)=\lcm\{ \deg(\overline \Theta, x): x\in \overline \Theta^{-1}(p)\}
\end{equation*} 
for  $p\in S^2$. 

\end{enumerate}
\end{lemma}

\begin{proof} 
  Let $f$, $\overline \Theta$, and $\overline A$ be as in the
  statement. Then $\overline A$ and $\overline{\Theta}$ are
  branched covering maps. In particular,  $\overline{\Theta}$ is surjective
  and open. Since
  $f\circ \overline{\Theta} =\overline \Theta\circ \overline A$, it  follows from Lemma~\ref{lem:f_desc_cont} that
  $f$ is continuous.  
 Then 
  Lemma~\ref{lem:2_3_branched}~\ref{item:2_out3_1} and
  \ref{item:2_out_3_2} 
  imply that  $f$ is actually  a branched covering map. 
  
  Note that    
  \begin{equation*}
    \deg(f)\cdot \deg(\overline \Theta)
    =\deg(f \circ \overline\Theta)
    =\deg(\overline \Theta\circ \overline A)
    =\deg(\overline\Theta)\cdot \deg(\overline A), 
  \end{equation*}
  and so
  \begin{equation*}
   \deg(f)= \deg(\overline A)\ge 2, 
  \end{equation*} 
  as claimed.
To show that $f$ is a Thurston map without periodic critical points, we first establish \ref{item:simobii} and \ref{item:simobiii}.

\smallskip 
\ref{item:simobii} 
Let
$V_{\overline \Theta}\coloneqq\overline \Theta(\crit(\overline \Theta))$ denote the set of critical values of
$\overline \Theta$. Since $T^2$ is compact and
the set of critical points of $\overline \Theta$ 
has no limit point 
in $T^2$, there are only finitely many critical points of
$\overline \Theta$. Hence the set $V_{\overline \Theta}$ is also  finite. We will first  prove
that $\post(f)\sub  V_{\overline{\Theta}}$.

Let  $p\in \post(f)$ be arbitrary. Then by  \eqref{eq:critpfn2}
the point $p$ is a critical value of some iterate of $f$. So   there exist $n\in \N$ and $q\in S^2$ with $\deg(f^n, q)\ge 2$ and $f^n(q)=p$. As a branched covering map, $\overline\Theta$ is surjective, and so we can find  $x\in T^2$ with $\overline \Theta(x)=q$.

Recall that by the Riemann-Hurwitz formula \eqref{eq:Riemann-Hurwitz}
the map $\overline{A}$ cannot have critical points, and hence is
locally injective. In particular,  $\deg(\overline{A}^n,x)=1$.   Since  $f\circ
\overline{\Theta}= \overline{\Theta}\circ \overline{A}$, we have 
 $f^n\circ
\overline{\Theta} = \overline{\Theta}\circ \overline{A}^n$. It follows that 
\begin{align*}
\deg(\overline{\Theta},\overline{A}^n(x))&=
 \deg(\overline{\Theta},\overline{A}^n(x))\cdot \deg(\overline{A}^n,x)\\
 &=\deg(\overline{\Theta}\circ \overline{A}^n,x)= \deg(f^n\circ \overline{\Theta},x)\\
  &= \deg(f^n,q)\cdot \deg(\overline{\Theta},x)\ge 2. 
  \end{align*}
  Thus  $\overline{A}^n(x)$ is a
critical point of $\overline{\Theta}$. So we have 
\begin{equation*}
 p=f^n(q)= (f^n\circ \overline{\Theta})(x) =(\overline{\Theta}\circ \overline{A}^n)(x)\in V_{\overline{\Theta}}. 
\end{equation*}
The desired inclusion $\post(f)\sub V_{\overline\Theta}$
follows.

To show  that actually $\post(f)= V_{\overline \Theta}$, we argue by contradiction and assume 
that there  exists  a point $p\in V_{\overline \Theta}\setminus \post(f)$. 
Then   by  \eqref{eq:critpfn2} the point  $p$ is not a critical value of any iterate $f^n$ of $f$.
Since $p$ is a critical value of $\overline \Theta$, the set $\overline \Theta^{-1}(p)$ contains a critical point $c$ of $\overline \Theta$. This implies that for each $n\in \N$, the set 
$\overline{A}^{-n}(c)$ consists of critical points of $\overline \Theta$. Indeed, if 
$a\in \overline A^{-n}(c)$, then 
$$f^n(\overline \Theta(a))=\overline \Theta(\overline A^n(a))=\overline \Theta(c)=p, $$ and so $\deg(f^n, \overline \Theta(a))=1$; moreover, $\overline A^n(a)=c$, and so 
\begin{align*} \deg(\overline \Theta, a)&=\deg(f^n, \overline \Theta(a))\cdot \deg(\overline \Theta, a)\\ 
&=\deg(f^n\circ \overline \Theta, a)=\deg(\overline \Theta\circ  \overline A^n, a)\\
&=\deg(\overline\Theta, \overline A^n(a))\cdot  \deg(\overline A^n, a)=\deg(\overline\Theta, c)\ge 2. \end{align*}
Since $\overline A$ is a covering map with $\deg(\overline A)\ge 2$, we have 
$$ \#\overline A^{-n}(c)=\deg(\overline A)^n\ge 2^n,$$
and so  there must be at least $2^n$ distinct  critical points of $\overline \Theta$. Since $n\in \N$ was arbitrary, and the number of critical points of $\overline \Theta$ is finite, this is a contradiction showing $\post(f)=V_{\overline \Theta}$.

Since  $\post(f)=V_{\overline\Theta}$ is
a finite set,  we conclude that $f$ is a Thurston map.   

\smallskip 
\ref{item:simobiii} 
We define 
$$ \nu(p)\coloneqq\lcm\{ \deg(\overline \Theta, x): x\in \overline \Theta^{-1}(p)\}$$ 
for $p\in S^2$. We claim that the function $\nu$ is equal to the
ramification function $\alpha_f$ of $f$. To prove this claim,  we will
show that $\nu$ has the  characterizing properties \ref{item:weight1}  and 
\ref{item:weight2} of $\alpha_f$ in 
Proposition~\ref{prop:weightf}. 

To see this,  let  $p\in S^2$ be   arbitrary, and $x\in \overline \Theta^{-1}(p)$.
If $y\coloneqq \overline A(x)$,  then 
$$\overline \Theta(y)=\overline \Theta(\overline A(x))=f(\overline \Theta(x))=f(p).$$
Hence $y\in \overline \Theta^{-1}(f(p))$, and  
\begin{align*} \deg(f,p)\cdot 
 \deg(\overline\Theta, x)&  = \deg(f\circ \overline \Theta, x)= \deg(\overline \Theta\circ \overline A, x)\\&=  \deg(\overline \Theta, \overline A(x))\cdot \deg( \overline A, x)= \deg(\overline \Theta, y). 
\end{align*} 
This shows that $\deg(f,p)\cdot 
 \deg(\overline\Theta, x)$  divides $\nu(f(p))$. Since this is true for all $x\in  \overline \Theta^{-1}(p)$ we conclude that 
$ \deg(f,p)\cdot \nu(p)$ divides $\nu(f(p))$.  So  the function $\nu$ satisfies condition \ref{item:weight1} in 
Proposition~\ref{prop:weightf}.

Now suppose  $\beta\colon S^2\to \widehat{\N}$ is another  function such that
$\deg(f,p)\cdot \beta(p)$ divides $\beta(f(p))$ for each $p\in S^2$. 
Then  $\deg(f^n,q)\cdot \beta(q)$ divides $\beta(f^n(q))$ for all 
$q\in S^2$ and $n\in \N$ (see the remarks after
Proposition~\ref{prop:weightf}).

Let $p\in S^2$ and $x\in \overline \Theta^{-1}(p)$ be
arbitrary. Then $\#\overline A^{-n}(x)=\deg(\overline A)^n\ge 2^n$. Since
there are only finitely many critical points of $\overline{\Theta}$,
there exist $n\in \N$ and $y\in \overline A^{-n}(x)$ such that
$y\notin \crit(\overline{\Theta})$. Let  
$q \coloneqq \overline \Theta (y)$. Then
\begin{equation*}
  f^n(q)=f^n(\overline\Theta(y))=\overline \Theta (\overline
A^n(y))=\overline \Theta (x)=p, 
\end{equation*}
and  
\begin{align*} 
  \deg(\overline\Theta, x)
  &=\deg(\overline\Theta, x)\cdot \deg(\overline A^n, y) =\deg(\overline \Theta\circ \overline A^n, y)\\
  &= \deg(f^n\circ \overline \Theta, y)= \deg(f^n, q)\cdot
  \deg(\overline \Theta, y)\\
  &=\deg(f^n, q).
\end{align*} 
Clearly, $\deg(\overline\Theta, x)=\deg(f^n,q)$ divides $\deg(f^n,q)\cdot\beta(q)$, which in turn
divides $\beta(p)=\beta(f^n(q))$ by the remark above.   
Hence $\deg(\overline\Theta,
x)|\beta(p)$ for all $x\in \overline \Theta^{-1}(p)$. By definition
of $\nu$ this implies that $\nu(p)| \beta(p)$ for $p\in
S^2$. This means that $\nu$ satisfies  condition \ref{item:weight2} in 
Proposition~\ref{prop:weightf}. 

{}From the uniqueness property of $\alpha_f$ given 
by Proposition~\ref{prop:weightf} we conclude  $\nu =\alpha_f$ as desired. 

\smallskip
\ref{item:simobi}
We have already seen that $f$ is a Thurston map with $\deg(f)=\deg(\overline A)$. From 
\ref{item:simobiii} it follows that $\alpha_f(p) < \infty$ for all
$p\in S^2$. Thus $f$ has no  periodic critical points (see
Proposition~\ref{prop:otherramprops}~\ref{item:rami_infty}). 
\end{proof}

\begin{lemma}[Criterion for  parabolicity]
  \label{lem:paratorusquot} 
Let   $f\: S^2\ra S^2$ be   a quotient
  of a torus endomorphism and $\overline \Theta\: T^2\ra S^2$ be   as in
 Definition~\ref{def:quotient_torus}. Then   $f$ has a  parabolic orbifold if  
and only if   \begin{equation}\label {item:simobiv}
\deg(\overline \Theta, x)=\deg(\overline \Theta, y)\end{equation} 
for all $x,y\in T^2$ with $\overline \Theta(x)=
\overline \Theta(y)$. 
\end{lemma}

We do not know whether  condition  \eqref{item:simobiv} is always true, or equivalently, whether 
 every quotient of a torus endomorphism  has 
 a parabolic orbifold. One can  show this under the additional assumption that the map  is expanding. The proof is rather involved, and so we will not discuss it.

\begin{proof}  As in
 Definition~\ref{def:quotient_torus},  
let   $\overline{A}\colon T^2\to T^2$ be a torus endomorphism 
with  $\deg(\overline A)\ge 2$ for our given maps $f$ and $\overline \Theta$.

Suppose first that $\deg(\overline \Theta, x)=
\deg(\overline \Theta, y)$, whenever $x,y\in T^2$ and $\Theta(x)=\Theta(y)$. This implies   that  for arbitrary $p\in S^2$ the local degree of $\overline \Theta$ is the same   for each point in   $\overline \Theta^{-1}(p)$. Then by Lemma~\ref{lem:simobserv}~\ref{item:simobiii}  we have  $\alpha_f(p)=\deg(\overline \Theta, x)$, whenever $x\in \overline \Theta^{-1}(p)$. If $y\coloneqq \overline A(x)$, then $\overline \Theta(y)=\overline \Theta(\overline A(x))=f(\overline \Theta(x))= f(p)$, and so 
\begin{align*}
 \alpha_f(f(p))&=\deg(\overline \Theta, y)=\deg(\overline \Theta, y)\cdot \deg(\overline A, x)\\
 &= \deg(\overline \Theta\circ \overline A, x)=  \deg(f \circ \overline \Theta, x)\\
 &= \deg(f, p)\cdot \deg(\overline \Theta, x)=  \deg(f, p)\cdot\alpha_f(p). 
 \end{align*} It follows that  $\mathcal{O}_f=(S^2,
 \alpha_f)$   is parabolic by the condition in  Proposition~\ref{prop:parabolicOf}~\ref{item:Of_para3}. 
 
 Conversely, suppose that $f$ has a parabolic orbifold. 
  We claim  that the local degree of $\overline \Theta$ is constant in each fiber over a point in $S^2$.  For this it 
 suffices to show that if $p\in S^2$ and $x\in \overline \Theta^{-1}(p)$, then $\deg(\overline \Theta, x)=\alpha_f(p)$. 
 
 Note that the set $\overline \Theta^{-1}(\post(f))$ is finite. So by 
 picking $n\in \N$ large enough,
  we can find a point $y\in \overline A^{-n}(x)$ with $q\coloneqq
  \overline \Theta(y)\notin \post(f)$. Then $\alpha_f(q)=1$ and
   $\deg(\overline \Theta, y)=1$ by Lemma~\ref{lem:simobserv}~\ref{item:simobii}. We also have 
  $$f^n(q)=f^n(\overline \Theta(y))=\overline \Theta(\overline A^n(y))=\overline \Theta(x)=p,  $$
 and 
 \begin{align*}
   \deg(\overline \Theta, x)
   &=   \deg(\overline \Theta, x) \cdot \deg(\overline A^n, y)  
   \\
   &=  \deg(\overline \Theta \circ \overline A^n, y)
     = 
     \deg(f^n \circ \overline \Theta, y)
   \\
   &= \deg(f^n, q)\cdot \deg (\overline \Theta, y)
     = 
     \deg(f^n, q). 
  \end{align*} 
  The parabolicity  of $\mathcal{O}_f$ implies that 
  $$\alpha_f(p)=\alpha_f(f^n(q))=\alpha_f(q)\cdot \deg(f^n, q)=
  \deg(f^n, q). $$
  We conclude that 
  $$  \deg(\overline \Theta, x)=  \deg(f^n, q)=\alpha_f(p)$$
  as desired. 
\end{proof}

To complete the proof of Theorem~\ref{thm:Lattesstruc}, and to
establish the remaining implication \ref{item:Lattesruciv}
$\Rightarrow$~\ref{item:Lattessrucii}, let
$f\colon \CDach \to \CDach$ be given as in
\ref{item:Lattesruciv} with corresponding maps
$\overline{A}\colon \T\to \T$ and
$\overline{\Theta}\colon \T\to \CDach$ that  are holomorphic and
defined on a complex torus $\T$. Then $f$ is a quotient of a
torus endomorphism (see Definition~\ref{def:quotient_torus}),
and hence a Thurston map without periodic critical points by
Lemma~\ref{lem:simobserv}~\ref{item:simobi}. The equation
$f\circ \overline{\Theta} = \overline{\Theta} \circ
\overline{A}$ now implies
that $f$ is a holomorphic map and hence a rational map on
$\CDach$ (see Lemma~\ref{lem:2_3_branched}).

The universal cover of $\T$ (as  a Riemann surface) 
is $\C$ and so there exists a holomorphic covering map $\pi\: \C \ra  \T$. Actually, we can identify $\T$ with a quotient space $\C/\Gamma$, where $\Gamma$ is a suitable rank-$2$ lattice.
Under such an identification $\T \cong \C/\Gamma$, the map 
 $\pi\: \C \ra \C/\Gamma\cong \T$ is just the quotient map.

The map $\overline{A}$ can be lifted by $\pi$ to a homeomorphism $A\colon \C\to
\C$  such that
$\overline{A}\circ\pi = \pi\circ A$ (see 
Section~\ref{sec:applifttorend} and in particular
Lemma~\ref{lem:torilifts}~\ref{item:tori2}). The map  $A$ is holomorphic, because locally it can be written as $A=\pi^{-1}\circ \overline A \circ \pi$ for a suitable holomorphic branch of $\pi^{-1}$. Hence $A$ has to be  of the form $A(z)=\alpha z+\beta$
with $\alpha, \beta\in \C$, $\alpha\ne 0$.  Then 
we have the following commutative diagram:

\begin{equation}
  \label{eq:Athetafpi23}
  \xymatrix{ \C \ar[r]^{A} \ar[d]_{\pi } &
    \C \ar[d]^{\pi }\\
    \T \ar[r]^{\overline A} \ar[d]_{\overline\Theta} &
    \T \ar[d]^{\overline\Theta}
    \\
    \CDach \ar[r]^f & \CDach\rlap{.} 
  }
\end{equation}

We are now ready to show the implication \ref{item:Lattesruciv}
$\Rightarrow$ \ref{item:Lattessrucii} of
Theorem~\ref{thm:Lattesstruc}. As we will see,  the proof strongly relies on
the fact that the involved maps are holomorphic.

\begin{proof}[Proof of \ref{item:Lattesruciv}
  $\Rightarrow$~\ref{item:Lattessrucii} in 
  Theorem~\ref{thm:Lattesstruc}]
  Suppose $f\: \CDach \ra \CDach$ is a map as in statement
  \ref{item:Lattesruciv} of Theorem~\ref{thm:Lattesstruc}. Then
  there is a complex torus $\T$, a holomorphic torus endomorphism
  $\overline{A}\colon \T\to \T$ with $\deg(\overline A)>1$, and a non-constant holomorphic map
  $\overline{\Theta}\colon \T\to \CDach$ such that
  $f\circ\overline{\Theta} = \overline{\Theta}\circ
  \overline{A}$.
  As we discussed, $f$ is also a holomorphic map and hence a
  rational map on $\CDach$.  Moreover, $f$ is a quotient of a
  torus endomorphism and so 
  a Thurston map without periodic critical points.

  It remains to show that the orbifold of $f$ is parabolic. By   Lem\-ma~\ref{lem:paratorusquot} it is
  enough to prove that the local degree 
  of $\overline \Theta$ is constant in each fiber $\overline
  \Theta^{-1}(p)$, $p\in \CDach$.  We argue by contradiction and
  assume that there exist $p\in \CDach$ and $x,y\in \overline
  \Theta^{-1}(p)$ with
  \begin{equation}
    \label{eq:degdifftheta}
    \deg(\overline \Theta, x)\ne \deg(\overline \Theta,y). 
  \end{equation}
  In particular, one of these degrees must be $\ge 2$; so $p$ is a
  critical value of $\overline \Theta$ and hence belongs to $\post(f)$
  by Lemma~\ref{lem:simobserv}~\ref{item:simobii}.

  For all $n\in \N$ we have $\overline A^n(x), \overline A^n(y)\in
  \overline \Theta^{-1}(f^n(p))$, and, since $\overline A$ does not
  have critical points,
  \begin{align*}  
    \deg(\overline \Theta, \overline A^n(x))& = \deg(\overline \Theta,
    \overline A^n(x))\cdot \deg(\overline A^n, x)= \deg(\overline
    \Theta\circ \overline A^n, x)\\ &=\deg(f^n \circ \overline
    \Theta,x)= \deg(f^n,p) \cdot \deg(\overline \Theta, x)\\ & \ne
    \deg(f^n,p) \cdot \deg(\overline \Theta, y) = \deg(\overline
    \Theta, \overline A^n(y)).
\end{align*} 
In other words, the iterates  $\overline A^n(x)$ and $\overline
A^n(y)$ lie in the fiber over the point $f^n(p)\in \post(f)$ and
$\overline \Theta$ has different local degrees at these points. Since
there are only finitely many points in $\post(f)$, and each fiber
$\overline \Theta^{-1}(p)$ contains only finitely many points, the
points $x$ and $y$ must be preperiodic under iteration of $\overline
A$, and $p$ must be preperiodic under iteration of $f$. This implies
that in \eqref{eq:degdifftheta} we may in addition assume that $x$ and
$y$ are periodic points for $\overline A$, and that $p$ is a periodic
point for $f$.  Moreover, by replacing the maps $\overline A $ and $f$
with  suitable iterates, we are further reduced to the case where $x$ and
$y$ are fixed points of $\overline A$, and $p$ is a fixed point of
$f$.

If we introduce a suitable holomorphic coordinate $w$  near $p$ 
such that $w=0$ corresponds to  $p$, then $f$ has a local 
power series representation of the form 
$$ f(w)=\lambda w^d+\dots, $$
where $\lambda\ne 0$ and $d\in \N$.  Since $f$ has no periodic critical points by 
Lemma~\ref{lem:simobserv}~\ref{item:simobi}, we actually have  $d=1$, and so 
$$ f(w)=\lambda w+\dots \,.$$

As we discussed before the proof,  the
map $\overline A$ lifts to a map $A\: \C\ra \C$ of the form
$A(z)=\alpha z+\beta$ for $z\in \C$, where $\alpha, \beta\in \C$,
$\alpha\ne 0$,  such that we have a commutative diagram as    in \eqref{eq:Athetafpi23}.
 This implies that by introducing a  suitable
holomorphic coordinate $u$  near  $x$ such that $u=0$ corresponds to
$x$, the maps $\overline A$ and $\overline \Theta$  can be given 
the forms  
$ \overline A(u)=\alpha u$ and 
$$ w=\overline \Theta(u)= b u^k+\dots$$
near $u=0$, where $b\ne 0$ and $k=\deg(\overline \Theta, x)$.  Similarly, by using a suitable  holomorphic coordinate $v$  near $y$ such that $v=0$ corresponds to $y$, we can write 
$ \overline A(v)=\alpha v$ and 
$$ w=\overline \Theta(v)= c v^n+\dots $$ 
near $v=0$, where $c\ne 0$ and $n=\deg(\overline \Theta, y)$.  
Since $f\circ \overline \Theta=\overline \Theta \circ   \overline A$ near $u=0$,  we obtain 
$$ 
f(\overline \Theta(u))= \lambda b u^{k}+\dots= \overline \Theta( \overline A(u))= b \alpha^k u^k+\dots\,.
$$ 
In particular,  $\lambda= \alpha^k$. Similarly, by considering the relation 
$ f(\overline \Theta(v)) =\overline \Theta( \overline A(v))$ near $v=0$, we obtain $\lambda= \alpha^n$. We conclude that 
$ \alpha^k=\lambda=\alpha^n$. 
Now 
$ 2\le \deg(f)=\deg (\overline A)=|\alpha|^2$ by \eqref{eq:degf_a2}, and so $|\alpha|>1$; but then $ \alpha^k=\alpha^n$ implies  $k=n$. 
This contradicts our assumption that 
$ k= \deg(\overline \Theta, x)\ne \deg(\overline \Theta, y)=n. $ 
 \end{proof}

This finishes the proof of Theorem~\ref{thm:Lattesstruc}. 

\section{Classifying Latt\`{e}s maps}
\label{sec:clas-latt-maps}
\index{Latt\`{e}s map!classification}

Theorem~\ref{thm:Lattesstruc} allows us  to  explicitly construct
 each  Latt\`{e}s map as a quotient of  a holomorphic automorphism  $A\colon \C\to \C$ by  a crystallographic group $G$. For such a map $A$ to pass to the quotient $\C/G$ it has to be $G$-equivariant. In this section 
we study  this   condition on $A$ and some related questions in more detail. 
Our results essentially provide a   classification of all Latt\`es maps.

Let $G$ be a crystallographic group not isomorphic to $\Z^2$, and let 
$\Theta\colon \C\to \CDach$ be a holomorphic map induced by $G$
as provided by Proposition~\ref{prop:G_yields_T}. Let $A\colon \C\to \C$
be a map of the form $A(z)= \alpha z + \beta$, where
$\alpha,\beta\in \C$ with $\abs{\alpha}>1$. Then by Lemma~\ref{lem:f_groupdescend} there is a
(unique) map
$f\colon \CDach \to \CDach$ that satisfies
$f\circ \Theta = \Theta \circ A$  if
and only if $A$ is $G$-equivariant. 
In this case,   $f$ is a Latt\`{e}s map
by condition~\ref{item:Lattessruciii} in Theorem~\ref{thm:Lattesstruc}.  
\index{G-equivariant@$G$-equivariant}  
\index{equivariant}
\index{group action!map equivariant under}

Recall from Proposition~\ref{prop:G_yields_T} that for a given
crystallographic group $G$, the map $\Theta$ is unique up  to
postcomposition with a M\"{o}bius transformation. So suppose 
$\widetilde{\Theta} = \varphi\circ \Theta$ is  another
(holomorphic) map induced by $G$, where
$\varphi\colon \CDach \to \CDach$ is a M\"{o}bius
transformation. Then if  $f\circ \Theta = \Theta \circ A$ it is
immediate to check that
$\widetilde{f}\coloneqq \varphi\circ f \circ \varphi^{-1}$ is the
unique map that satisfies
$\widetilde{f}\circ \widetilde{\Theta} = \widetilde{\Theta} \circ
A$.
Thus the Latt\`{e}s map induced by a map $A$ that is equivariant
for a given crystallographic group $G$ is unique up to
conjugation by a M\"{o}bius transformation.

For a given Latt\`{e}s map $f\colon \CDach \to \CDach$, the map
$A\colon\C\to \C$ and the group $G$ provided by
Theorem~\ref{thm:Lattesstruc}~\ref{item:Lattessruciii} are not
unique. Indeed, we can conjugate $G$ and $A$ by an arbitrary map
$h\in \Aut(\C)$.  Then
$\widetilde G= \{ h^{-1}\circ g \circ h : g\in G\}$ is also a
crystallographic group. If
$\widetilde{A}= h^{-1} \circ A \circ h$ and
$\widetilde{\Theta} = \Theta\circ h$, then we obtain the same map
$f$ in Theorem~\ref{thm:Lattesstruc}~\ref{item:Lattessruciii}, if
we replace $G,\Theta, A$ with
$\widetilde G, \widetilde\Theta, \widetilde A$, respectively. The
situation is illustrated in the following commutative diagram:
\begin{equation}
  \label{eq:LattesGnotunique}  
  \xymatrix{
    \C \ar[r]^{\widetilde{A}} \ar[d]_h \ar@/_2pc/[dd]_{\widetilde{\Theta}} 
    & \C \ar[d]^h\ar@/^2pc/[dd]^{\widetilde{\Theta}}
    \\
    \C \ar[r]^{A}\ar[d]_{\Theta} & \C\ar[d]^{\Theta}
    \\
    \CDach \ar[r]^{f} & \CDach\rlap{.}
    }
\end{equation}

By using such a conjugation,  in 
Theorem~\ref{thm:Lattesstruc}~\ref{item:Lattessruciii}
 we can always assume 
 that the group $G$ is one of the groups
$\widetilde G$ listed in Theorem~\ref{thm:G_signature}.


The requirement that the map $A(z)=\alpha z+ \beta$ is $\widetilde G$-equivariant 
can then be explicitly analyzed and  puts strong restrictions on $\alpha$
and $\beta$ as the following proposition shows.

\begin{prop}
  \label{prop:alph_beta_G}
  Let $G=\widetilde G$ be a crystallographic group as in Theorem~\ref{thm:G_signature} not isomorphic
  to $\Z^2$. Let
  $\Gamma=\Z\oplus \Z\tau$ be the underlying lattice, where
  $\tau\in \C$ with $\imag(\tau)>0$ in the case $(2222)$,
  $\tau=\iu$ in the case $(244)$, and $\tau=\omega=e^{\pi \iu/3}$
  in the cases $(333)$ and $(236)$. Then $A\colon \C\to\C$ given
  by $A(z)=\alpha z+ \beta$ (where $\alpha, \beta\in \C$,
  $\alpha\neq 0$) is $G$-equivariant if and only if
  \index{equivariant}\index{G-equivariant@$G$-equivariant}\index{group action!map equivariant under}
   \begin{equation*}    
    \left. 
      \begin{array}{rl} 
        \alpha,\alpha  \tau,&\!\!\! \!         2\beta
        \\ 
        \alpha,
        & \!\!\! \!     (1+\iu)\beta
        \\
        \alpha,
        & \!\!\! \!     (1+\omega)\beta
        \\
        \alpha,
        & \!\!\! \!    \beta
      \end{array}
       \right\}
    \text{are elements of $\Gamma$ when $G$ is of type }
    \left\{\!
      \begin{array}{c} 
        \textnormal{(2222)},
        \\   
        \textnormal{(244)}, 
        \\ 
        \textnormal{(333)},
        \\
        \textnormal{(236).}
      \end{array} 
    \right.
  \end{equation*}
  \end{prop}

As we will see in the proof,  the condition on $\alpha$ is equivalent to the
requirement that $\alpha\Gamma \subset \Gamma$.
According to this proposition, we can always choose
$\alpha\in \Z\setminus\{0\}$ and $\beta=0$ for any lattice
$\Gamma$. This and
Theorem~\ref{thm:Lattesstruc}~\ref{item:Lattessruciii} imply that
Latt\`{e}s maps exist for all signatures $(2,2,2,2)$, $(2,4,4)$,
$(3,3,3)$, and $(2,3,6)$.  We will discuss more explicit examples
for each of these signatures later in
Section~\ref{sec:examples-lattes-maps}.

\begin{proof}
  Recall from \eqref{eq:equivariant} that a map
  $z\mapsto A(z)=\alpha z+\beta$ as in the statement is
  $G$-equivariant if and only if $A\circ g\circ A^{-1}\in G$ for
  all $g\in G$.
 
  The maps $g\in G$ have the form $g(z)=\lambda z+ \gamma$, where
  $\gamma\in \Gamma$ and $\lambda$ is a root of unity depending
  on the type of $G$.  An elementary computation shows that
  \begin{equation*}
    (A\circ g\circ A^{-1})(z)
    = 
    \lambda z + \alpha \gamma + (1-\lambda )\beta.
  \end{equation*}
  Using this first for $\gamma=0\in \Gamma$, we see that
  \eqref{eq:equivariant} can only be valid if
  \begin{equation}
    \label{eq:betainG}
    (1-\lambda)\beta\in \Gamma, 
  \end{equation}
  for the appropriate roots of unity $\lambda$; in addition, it
  is necessary that
  \begin{equation}
    \label{eq:alphainG}
    \alpha \gamma\in \Gamma  \text{ for all $\gamma\in \Gamma$.}
  \end{equation}

  Conversely, if \eqref{eq:betainG}
  and \eqref{eq:alphainG} are true, $A$ satisfies
  \eqref{eq:equivariant}.  
  Thus $A$ is
  $G$-equivariant if and only if $\alpha$  and $\beta$ satisfy
  \eqref{eq:betainG} and \eqref{eq:alphainG}.

  Note that \eqref{eq:alphainG} is equivalent to the condition
  that $\alpha\Gamma\subset \Gamma$. Since $1$ and $\tau$
  generate the lattice $\Gamma$, this in turn is the same as the
  requirement that 
  \begin{equation}
    \label{eq:eq:alphainG2}
    \alpha\in \Gamma \text{ and } \alpha \tau\in \Gamma.
  \end{equation} 
 
  If $G$ is of type $(244)$,
  $(333)$, or  $(236)$, then we can omit the second condition here. Indeed, in these cases  $\tau\Gamma= \Gamma$, and so  $\alpha\in
  \Gamma$ if and only if  $\alpha \tau \in \Gamma$. 

 This discussion shows that  $A$ is $G$-equivariant if and only if $\alpha$ 
 satisfies the conditions as in 
  the statement of the proposition, and $\beta$
  satisfies~\eqref{eq:betainG}.    
  To analyze this latter condition further, we consider several cases
  depending on the type of $G$.  

\smallskip 
If $G$ is of type $(2222)$, then 
we have $\lambda=\pm 1$. Thus
  \eqref{eq:betainG} is true if and only if $2\beta\in \Gamma$. 
  
If $G$ is of type  (244), then  we have $\Gamma=\Z\oplus \Z\iu$ and 
 $\lambda= 1,\iu, -1, -\iu$. Thus
  \eqref{eq:betainG} implies that
     $(1+\iu)\beta\in \Gamma$. 

Conversely, suppose this last condition   is true.  Note that $-\iu\Gamma=\Gamma$ and 
 $(1-\iu)\Gamma\subset \Gamma$.  So we conclude   that  $-\iu(1+\iu)\beta= (1-\iu)\beta\in \Gamma$  and 
  $(1-\iu)(1+\iu)\beta= 2\beta\in \Gamma$. Therefore,  \eqref{eq:betainG}
  holds for  $\lambda=1,\iu, -1, -\iu$.

  If $G$ is of type  (333), then  we have $\Gamma=\Z\oplus \Z\omega$
   and
   $\lambda=1,
  \omega^2,\omega^4$ (recall that $\omega=e^{\pi\iu/3}$). So
  \eqref{eq:betainG} implies that
  $(1-\omega^4)\beta=(1+\omega)\beta\in \Gamma$. 

Conversely, if this last condition is true, then using 
  $\omega\Gamma=\Gamma$ we obtain  
     $\omega^5(1+\omega)\beta=
  (1-\omega^2)\beta\in \Gamma$;  so 
   \eqref{eq:betainG} holds for
  $\lambda=\omega^j$ where $j=0,2,4$.

  Finally, if $G$ is of type (236), then
  $\Gamma=\Z\oplus \Z\omega$ and $\lambda=\omega^j$ where
  $j=0,\dots,5$. Thus \eqref{eq:betainG} implies
  $(1-\omega^5)\beta=(1+\omega^2)\beta\in\Gamma$. Note that
  $1+\omega^2=\omega$, and so  we obtain $\omega\beta\in
  \Gamma$. Since $\omega\Gamma =\Gamma$, this shows that
  $\beta\in \Gamma$. 
  
  Conversely, assume that $\beta\in \Gamma$. Using
  $\omega\Gamma= \Gamma$ again, we conclude that
  $\omega^j\beta\in \Gamma$ for  $j=0,\dots,5$. This implies
  $\beta -\omega^j\beta = (1-\omega^j)\beta\in \Gamma$. Thus
  \eqref{eq:betainG} is satisfied for all $\lambda=\omega^j$
  with  $j=0,\dots,5$.
  
  The claim follows.
\end{proof}

Suppose a Latt\`{e}s map $f\colon \CDach\to \CDach$ is given as
in Theorem~\ref{thm:Lattesstruc}~\ref{item:Lattessruciii}. Then
$f\circ \Theta= \Theta\circ A$, where $\Theta\colon \C\to \CDach$
is induced by a crystallographic group $G$ (not isomorphic to
$\Z^2$) and $A(z) =\alpha z + \beta$ (with $\alpha,\beta\in \C$
and $\abs{\alpha}>1$) is $G$-equivariant. As we discussed, here
we can always assume that $G=\widetilde{G}$ is as in
Theorem~\ref{thm:G_signature}.

One can make another reduction. Namely, we can replace $A$ with
any map $\widetilde A=g\circ A$, where $g\in G$. Indeed, we then
have $\Theta \circ \widetilde A= \Theta \circ A = f\circ \Theta$,
because $\Theta$ is induced by $G$ and so $\Theta\circ
g=\Theta$.
In particular, $\widetilde A$ is also $G$-equivariant and induces
the same Latt\`{e}s map $f$.

One can use this remark to substantially  restrict  the values of the coefficient 
$\beta$ of the map $A$. 

\begin{prop}
  \label{prop:alpha_beta2}
  Let $f\colon \CDach \to \CDach$ be a Latt\`{e}s map as above obtained
  from a map $A\colon \C\to\C$ given by $A(z) = \alpha z +
  \beta$, a crystallographic group $G=\widetilde G$  as in
Theorem~\ref{thm:G_signature}, and a holomorphic map
  $\Theta\colon \C\to \CDach$ induced by $G$. 
  Then by postcomposing $A$ with a suitable
  translation $g\in G_{\textnormal{tr}}\subset G$ we
  can always assume that $\beta$ has one of the following forms: 
  \begin{equation*}    
    \left. 
      \begin{array}{l}
        \beta\in \{0, \tfrac 12, \tfrac 12 \tau , \tfrac 12 (1 + \tau)\}
        \\ \rule{0pt}{2.6ex}    
        \beta\in \{0,\tfrac 12(1+\iu)\}
        \\ \rule{0pt}{2.6ex}    
        \beta\in \{0, \tfrac{1}{3} + \tfrac{1}{3}\omega, 
        \tfrac{2}{3} + \tfrac{2}{3}\omega \}
        \\ \rule{0pt}{2.6ex}    
        \beta=0
      \end{array}
    \right\}
    \text{ when $G$ is of type }
    \left\{\!
      \begin{array}{c} 
        \textnormal{(2222)},
        \\   \rule{0pt}{2.6ex}    
        \textnormal{(244)}, 
        \\ \rule{0pt}{2.6ex}    
        \textnormal{(333)},
        \\ \rule{0pt}{2.6ex}    
        \textnormal{(236).}
      \end{array} 
    \right.
  \end{equation*}
\end{prop}

\begin{proof} As before, we denote by $\Gamma$ the underlying lattice of  $G$.
 Let $\sim$ be the equivalence relation on $\C$ defined by
  $z\sim w$ if and only if $w-z\in \Gamma$ for $z,w\in \C$.
This is  the equivalence relation induced by the action of the
  subgroup of all translations $\Gtr\subset G$. As we have seen, if $g(z)=z+\gamma$ with $\gamma\in \Gamma$, then  $g\in \Gtr$ and we can replace the $G$-equivariant map 
  $A(z)=\alpha z+\beta$ with 
  $$ \widetilde A(z) =(g\circ A)(z)=\alpha z+ (\beta+\gamma).$$  
 This means that we can change $\beta$ to any element $\beta'$ with $\beta'\sim\beta$ without affecting the Latt\`es map $f$.  We now analyze this in combination with the condition  on $\beta$ in Proposition \ref{prop:alph_beta_G}  for the different types of $G$.

  If $G$ is of type $(2222)$, then $2\beta\in \Gamma$
  by Proposition~\ref{prop:alph_beta_G} or equivalently, $\beta\in \frac 12 \Gamma$. As we can replace $\beta$ with  any $\beta'$ satisfying $\beta'\sim \beta$, we  may assume that  
  \begin{equation}\label{eq:betared} 
    \beta = \tfrac12 (k +\ell \tau),\quad \text{where $k, \ell \in \{0,1\}$}. 
  \end{equation}
  So $\beta$ has the desired form. 

  If  $G$ is of type $(244)$, then  by 
  Proposition~\ref{prop:alph_beta_G} the relevant condition 
  is   $(1+\iu)\beta \in \Gamma$. This implies that 
  $$ \beta\in (1+\iu)^{-1} \Gamma=\tfrac12 (1-\iu)\Gamma\sub \tfrac12 \Gamma. $$  So again we may assume that $\beta$ 
 is  as in \eqref{eq:betared} with $\tau=\iu$.  For such $\beta$ we have 
  $$(1+\iu)\beta = \tfrac12 ( (k-\ell)+ (k+\ell)\iu) \in \Gamma$$
  precisely if $k=\ell\in \{0,1\}$. The statement follows in this   
  case.

  If $G$ is of type $(236)$, then by
  Proposition~\ref{prop:alph_beta_G} the condition on $\beta$ is
  $\beta\in \Gamma$, or equivalently $\beta\sim 0$. This means
  that we can always take $\beta=0$ in this case.
   
  Finally, if $G$ is of type $(333)$, then by
  Proposition~\ref{prop:alph_beta_G} the relevant condition is
  $(1+ \omega)\beta \in \Gamma$, where $\omega= e^{\pi\iu/3}$ and
  $\Gamma= \Z\oplus \Z\omega$.  Since $\omega\Gamma=\Gamma$ and
  $(1+\omega)(1+\omega^5)=3$, this implies
  $$ \beta \in (1+\omega)^{-1}\Gamma=\tfrac13 (1+\omega^5)\Gamma\sub \tfrac13 \Gamma.$$ 
  Hence we may assume that $\beta$ has the form
  \begin{equation*}
   \beta 
   = 
   \tfrac13 (k +\ell \omega),
   \quad 
   \text{where $k,\ell\in \{0,1,2\}$}.   
  \end{equation*}
  If we use $\omega^2=\omega-1$, we see that for such $\beta$ we
  have
  \begin{equation*}
    (1+\omega)\beta
    = 
    \tfrac13 ((k-\ell)+(k+2\ell)\omega)\in \Gamma
  \end{equation*}
  precisely if $k=\ell\in \{0,1,2\}$. Again $\beta$ can be given
  the desired form.
\end{proof}

\section{Latt\`{e}s-type maps}
\label{sec:lattes-type-maps}
\index{Latt\`{e}s-type map}

We now consider Latt\`{e}s-type maps $f\: S^2 \ra S^2$ as in  
Definition~\ref{def:Lattestype}. 
If $f$ is such a map,  then there exists a crystallographic group $G$ acting on $\R^2\cong \C$,
a $G$-equivariant  affine map $A\: \R^2 \ra \R^2$,  and a branched covering map 
$\Theta\:\R^2 \ra S^2$ induced by $G$ such that $f\circ
A=\Theta\circ A$. Then  $f$ is
continuous (see Lemma~\ref{lem:f_desc_cont}). 
Since $\Theta$ is induced by $G$, the quotient space $\R^2/G$ is homeomorphic to $S^2$ which implies that $G$ is not isomorphic to $\Z^2$.  

In explicit constructions of Latt\`{e}s-type maps one usually
turns this around and starts with a crystallographic group $G$
not isomorphic to $\Z^2$ and an $G$-equivariant affine map
$A\: \R^2\ra \R^2$. Then $S^2=\R^2/G$ is a topological
$2$-sphere and the quotient map
$\Theta\: \R^2\ra \R^2/G\cong S^2 $ is a branched covering map
induced by $G$. The $G$-equivariance of $A$ ensures that this
map descends to the quotient $\R^2/G\cong S^2$ and so there
exists a continuous map $f\: S^2\ra S^2$ such that
$f\circ A=\Theta\circ A$ (see Lemma~\ref{lem:f_groupdescend}). As the considerations below will show,
the additional condition that the linear part $L_A$ of $A$ (see
\eqref{eq:affine}) satisfies $\det(L_A)>1$ ensures that $f$ is a
Thurston map.

Indeed,  let $\Gtr$ be the subgroup of translations in $G$. We know that then 
$T^2=\R^2/\Gtr$ is a (topological) $2$-torus. We denote by  $\pi\: \R^2 \ra T^2=\R^2/\Gtr$   the quotient map.

The argument in the proof of the implication 
\ref{item:Lattessruciii} 
  $\Rightarrow$ \ref{item:Lattesruciv} in 
  Theorem~\ref{thm:Lattesstruc} (see Section~\ref{sec:cryst-groups-latt})  shows that  $A$ and $\Theta$ descend to maps 
  $\overline A$ and  $\overline \Theta$ on $T^2$. In  this proof the maps were assumed to be holomorphic, but this played no role in the existence proof for  
 $\overline A$ and  $\overline \Theta$. 
  So we obtain continuous maps $\overline A\: T^2\ra T^2$ and $\overline \Theta\: T^2\ra S^2$ such that 
  $\overline A \circ \pi= \pi \circ A$ and $\Theta=\overline \Theta \circ \pi$. Note that as a composition of the covering map $\pi
  \: \R^2 \ra T$ with the homeomorphism $A\: \R^2 \ra \R^2$, the map $ \pi \circ A\: \R^2 \ra T^2$ is  a covering map. 
This combined with 
the last relations implies  that  $\overline A$ and   $\overline
\Theta$ are branched covering maps (see
Lemma~\ref{lem:2_3_branched}~\ref{item:2_out_3_2}).
As we discussed, it follows that $\overline A$ is a (topological) torus endomorphism. 

Similar to \eqref{eq:holom_qu_diagram}, one can summarize the relations of these maps in the following commutative diagram: 

\begin{equation}
  \label{eq:Athetafpi34}
  \xymatrix{
    \R^2 \ar[r]^{A} \ar[d]_{\pi} \ar@/_2pc/[dd]_{\Theta} 
    & \R^2 \ar[d]^{\pi}\ar@/^2pc/[dd]^{\Theta}
    \\
    T^{\smash{2}} \ar[r]^{\overline{A}}\ar[d]_{\overline{\Theta}}  & T^{\smash{2}}\ar[d]^{\overline{\Theta}}
    \\
    S^2 \ar[r]^{f} & S^2\rlap{.}
  }
\end{equation}

Since $\overline \Theta$ and $\overline \Theta \circ \overline A$
are branched covering maps, and $f\circ \overline
\Theta=\overline \Theta\circ \overline A$, the map $f$ is also a branched covering map (see   Lemma~\ref{lem:2_3_branched}~\ref{item:2_out3_1} and~\ref{item:2_out_3_2}).   

It is easy to see that   $\deg(f)=\deg(\overline A)$ (see the beginning of the proof of Lemma~\ref{lem:simobserv}), but 
the degree of $f$  can also  be computed from $A$. 

\begin{lemma}\label{lem:deglatttype} 
Let $f\: S^2\ra S^2$ be a Latt\`es-type map, and suppose  
 $A\: \R^2\ra \R^2$ is  an affine map and $\overline A$  a torus endomorphism as in \eqref{eq:Athetafpi34}. Let $L_A$ be the linear part of $A$. 
 Then $\deg(f)=\deg(\overline A) =\det(L_A)$.  

In particular, if $f\: \CDach \ra \CDach$ is a Latt\`es map and $A(z)=\alpha z+\beta$ is as in Theorem~\ref{thm:Lattesstruc}~\ref{item:Lattessruciii}, then $\deg(f)=|\alpha|^2$.
\end{lemma} 

\begin{proof}
  Let $f\:S^2\ra S^2$ be a Latt\`es-type map, and suppose $G$ is
  a crystallographic group and $A$ an affine map as in
  Definition~\ref{def:Lattestype}. Then we have a commutative
  diagram as in \eqref{eq:Athetafpi34} and we know that
  $\deg(f)=\deg(\overline A)$. So we have to verify that
  $\deg(\overline A)=\det(L_A)$. Essentially, this follows from  standard facts about  degrees of torus endomorphisms as discussed in more detail in  Section~\ref{sec:applifttorend}.

  Indeed, let $\Gamma\sub \R^2\cong \C$ be  the underlying lattice of $G$.
  Then $G_{\text{tr}}$ consists of all translations of the form $u\in \R^2\mapsto 
  \tau_\gamma(u)\coloneqq u+\gamma$, where $\gamma\in \Gamma$. 
  Accordingly, we can identify the torus
  $T^2=\R^2/G_{\text{tr}}$ with the quotient $\R^2/\Gamma$ and can think of the lattice $\Gamma$ as representing 
  the fundamental group of $T^2$ (see the discussion after Lemma~\ref{lem:torilifts}). 

 Now  an elementary computation shows that 
  \begin{equation*}
    A\circ \tau_\ga\circ A^{-1} = \tau_{L_A(\ga)} 
  \end{equation*}
  for each $\gamma\in \R^2$, and in particular for each $\gamma\in \Gamma$. Since $A$ is a lift of $\overline A$ to $\R^2$, it follows that  $L_A$ is 
  the unique map induced by $\overline A$ on the fundamental group $\Gamma$ of $T^2$
  (see Lemma~\ref{lem:torilifts}~\ref{item:tori3}).
  Now  Lemma~\ref{lem:torilifts}~\ref{item:tori4} implies  
    $\deg(\overline{A})= \det(L_A)$ as desired. 

%
 
%

If $f$ is a Latt\`es map and $A(z)=\alpha z+\beta $ 
as in Theorem~\ref{thm:Lattesstruc}~\ref{item:Lattessruciii}, 
then in complex notation $L_A(z)=\alpha z$ for $z\in \C$. 
For the determinant of $L_A$ considered as a $\R$-linear map, we have $\det(L_A)=|\alpha|^2$. So it  follows from the first part of the proof that $\deg(f)=\det(L_A)=|\alpha|^2$ as claimed.  
 \end{proof}

 \begin{proof} [Proof of Proposition~\ref{prop:immediate}]
 Let  $f\: S^2\ra S^2$ be  a Latt\`es-type map, and $G$, $A$, and $\Theta$   be as  in Definition~\ref{def:Lattestype}. Then we have a diagram as in 
 \eqref{eq:Athetafpi34}. Here $\det(L_A)>1$ by assumption which by 
 Lemma~\ref{lem:deglatttype}  translates into 
 $\deg(f)=\deg(\overline A)=\det(L_A)\ge 2$. We conclude that 
 $f$ is a quotient of a torus endomorphism 
 (see Definition~\ref{def:quotient_torus}). 
 
 So we can apply  
Lemma~\ref{lem:simobserv} and it follows that $f$ is a Thurston map 
without periodic critical points. It remains to show that $f$ has a parabolic orbifold.

For this we verify the criterion in Lemma~\ref{lem:paratorusquot} with the branched 
covering map $\overline \Theta\: T^2\ra S^2$ as provided by  \eqref{eq:Athetafpi34}. So suppose $x,y\in T^2$ and 
$\overline \Theta(x)=\overline \Theta(y)$. Since $\pi\: \R^2\ra T^2$ is
surjective, there exist $u,v\in \R^2$ with 
$\pi(u)=x$ and $\pi(v)=y$. Then 
$$ \Theta(u)=(\overline \Theta \circ \pi)(u)=
\overline \Theta (x)= \overline \Theta (y)= (\overline \Theta \circ \pi)(v)=
 \Theta(v). $$ 
 Since $\Theta$ is induced by the crystallographic group $G$, there exists $g\in G$ with $v=g(u)$. Now $\Theta=\Theta \circ g$ and $$\deg(\pi, u)=\deg(\pi, v)=\deg(g, u)=1. $$ We conclude that  
 \begin{align*}
\deg(\overline \Theta , y)&=\deg(\overline \Theta , y)\cdot \deg(\pi, v)=
\deg(\Theta, v)\\
&=   \deg(\Theta, v)\cdot \deg(g, u) = \deg(\Theta, u) \\
&= \deg(\overline \Theta , x)\cdot \deg(\pi, u)=\deg(\overline \Theta , x)
  \end{align*}
  as desired. 
 \end{proof} 
 
 \begin{cor} \label{cor:orbofLattty}
 Let $f\: S^2\ra S^2$ be a Latt\`es-type map with a 
 crystallographic group $G$ and a branched covering map $\Theta\: \R^2\ra S^2$ induced by $G$ as in Definition~\ref{def:Lattestype}, 
 and let $\mathcal{O}_f=(S^2, \alpha_f)$ be the  associated orbifold of $f$. Then 
 for $u\in \R^2$ we have $$ \alpha_f(\Theta(u))=\deg(\Theta,u)=\#G_u.$$ 
   \end{cor}
 
 \begin{proof} Let $u\in \R^2$. Then  $\deg(\Theta, u)=\#G_u$ for $u\in \R^2$  as follows from the second equality in \eqref{eq:degT_orderG} and the uniqueness statement in Proposition~\ref{prop:G_yields_T} (it is also easy to see this directly). 
 
 If we use the notation as in  the previous proof, then the considerations there show  that  if $p\coloneqq \Theta(u)$, then the degree of $\overline \Theta$ in each point of  the fiber
  $ \overline \Theta^{-1}(p)$ is the same and is equal to 
   $\deg(\Theta, u)$. So by Lemma~\ref{lem:simobserv}~\ref{item:simobiii} we also have $\alpha_f(\Theta(u))=\deg(\Theta, u)$.  
    \end{proof} 
  
We know that the  type of a crystallographic group $G$ is determined by the orders of the point stabilizers $\#G_u$, $u\in \R^2$. Moreover, if $\Theta\: \R^2\ra S^2$ is induced by $G$, then the  map $Gu\in \R^2/G\mapsto \Theta(u)
\in S^2$ is a bijection (it is actually a homeomorphism; see the discussion after Proposition~\ref{prop:G_yields_T}). So it  follows from the corollary 
 that the  signature of 
  $\mathcal{O}_f$  corresponds to the type of the crystallographic
   group $G$. For example, if $G$ is of type $(2222)$, then the signature of  	$\mathcal{O}_f$ is $(2,2,2,2)$. 
   
   Of course, the corollary also applies when  $f\: \CDach\ra \CDach$   is a Latt\`es map and $\Theta\: \C\ra \CDach$ is a holomorphic map induced by $G$. Then 
   the statement shows that $\alpha_f=\alpha$, where $\alpha$ is as in Proposition~\ref{prop:G_yields_T}, and that   $\Theta$ is the (holomorphic) universal orbifold 
   covering map of $(\CDach, \alpha_f)$.

Since a Latt\`{e}s-type map has parabolic orbifold and no periodic critical points, we know by
Proposition~\ref{prop:parabolicOf} that  the orbifold of each  such map   has one of
the signatures $(2,2,2,2)$, $(2,4,4)$, $(3,3,3)$, or
 $(2,3,6)$ (this also follows from Corollary~\ref{cor:orbofLattty}).  The last three cases  lead to nothing new and  essentially give   Latt\`es maps.   
 
 \begin{prop} \label{prop:rigidlattestype}
  Let $f\: S^2 \ra S^2$ be a Latt\`es-type map with orbifold signature
  $(2,4,4)$, $(3,3,3)$, or $(2,3,6)$. Then $f$ is topologically
  conjugate to a Latt\`es map. 
  \end{prop}
 
 To prove this statement we need a lemma that gives a criterion
 when an $\R$-linear map $L\:\C\ra \C$   is $\C$-linear.  Here
 the  $\R$-linearity or  $\C$-linearity for $L$ 
 of course means that $L(z+w)=L(z)+L(w)$ and $L(\lambda
 z)=\lambda L(z)$ for all $z,w\in \C$ and all $\lambda\in \R$ or
 all $\lambda\in \C$, respectively.   
 
   \begin{lemma} \label{lem:RlinClin} Let $L\: \C\ra \C$ be an $\R$-linear map with 
  $\det(L)>0$. Suppose there exist $\zeta\in \C\setminus \R$ and $\eta \in \C$ with $L(\zeta z)=\eta L(z)$ for all $z\in \C$. Then $L$ is $\C$-linear.
  \end{lemma}

  \begin{proof} Since $L$ is $\R$-linear, there exist unique numbers $a,b\in \C$ such that $L(z)=az+b \overline z$ for $z\in \C$. Then the determinant of $L$ (as an $\R$-linear map) is given by $\det(L)=|a|^2-|b|^2>0$. It follows that  $a\ne 0$. 
  
  Now for all $z\in \C$ we have 
  $$ L(\zeta z)=\zeta a z +\overline \zeta b \overline z=\eta L(z)=\eta az +\eta b \overline z, $$ and so 
  $$ \zeta a=\eta a \quad \text{and} \quad \overline \zeta b= \eta b. $$
  Since $a\ne 0$, the first equation implies  $\zeta=\eta$. Then the second equation combined with the fact that $\zeta\notin \R$ gives  $b=0$. 
  Hence $L(z)=az$ for $z\in \C$. This shows that $L$ is $\C$-linear. 
  \end{proof}

  \begin{proof}[Proof of Proposition~\ref{prop:rigidlattestype}]
    We know that there exists a crystallographic group $G$, a branched covering map $\Theta\: \R^2 \ra S^2$ induced by $G$, and  a $G$-equivariant affine homeomorphism
    $A\:\R^2 \ra \R^2$ with $\det(L_A)>1$ such that $f$ arises as
    in \eqref{eq:def_lattes_type}.  By conjugation with a
    suitable map in $\Aut(\C)$, we may assume that $G$ is one of
    the groups $\widetilde G$ listed in
    Theorem~\ref{thm:G_signature} (see the discussion preceding
    \eqref{eq:LattesGnotunique}).  Then $G$  is not isomorphic to $\Z^2$, because 
    $\C/G\cong S^2$.

     It follows from 
    Proposition~\ref{prop:G_yields_T} that we can find a
    homeomorphism $\varphi\: S^2 \ra \CDach$ such that
    $\varphi\circ \Theta\: \R^2 \cong \C \ra \CDach$ is a
    holomorphic map. So if we replace the original map $\Theta$ with  $\varphi\circ \Theta$ and $f$ with its conjugate 
    $\varphi \circ f \circ \varphi^{-1}$, then we are further reduced to the case that $S^2=\CDach$ and that $\Theta$ is holomorphic (a related argument was given in the beginning of Section~\ref{sec:clas-latt-maps}).  
    It is enough to show that then   $f\: \CDach \ra \CDach$ is a Latt\`es map.
    
   By assumption the signature of the orbifold of $f$ is $(2,4,4)$,
    $(3,3,3)$, or $(2,3,6)$. By Corollary~\ref{cor:orbofLattty}
    this means that $G$ is a crystallographic group $\widetilde G$ of type
    $(244)$, $(333)$, or $(236)$ in Theorem~\ref{thm:G_signature}.  In these cases, $G$ contains a rotation $g_0$
    of the form $z\in \C\mapsto g_0(z)=\zeta z$, where
    $\zeta=e^{2\pi \iu/n}$ is a
    primitive $n$-th root of unity with $n\in \{3,4,6\}$.  In
    particular, $\zeta \in \C\setminus \R$. Since the homeomorphism $A$ passes to the
    quotient $\widehat \C\cong \C/G$, this map is $G$-equivariant
     (Lemma~\ref{lem:f_groupdescend})
    and so there exists 
     $g_1\in G$ 
    such that
    \begin{equation} 
      \label{eq:AghA} 
      A\circ g_0= g_1\circ A. 
    \end{equation} 
    Let $L=L_A$ be the linear part of $A$. This is
    an $\R$-linear map satisfying $\det(L)>0$. Since $G$ consists of
    orientation-preserving isometries, the linear part of $g_1$ is $\C$-linear, as for
    every map in $G$. Comparing linear parts of the
    maps in \eqref{eq:AghA}, we see that there exists $\eta \in \C$,
    $|\eta|=1$, such that
    $$ L(\zeta z)=\eta L(z)$$ 
    for all $z\in \C$. This shows that $L$ satisfies the hypotheses of
    Lemma~\ref{lem:RlinClin} and we conclude that $L$ is $\C$-linear. 
    Hence $A$ can be written in the form $A(z)=\alpha z+\beta$ for $z\in \C$, where
    $\alpha, \beta\in \C$, $\alpha\ne 0$. In particular, $A$ is holomorphic, and it follows 
    that $f$ is indeed a Latt\`es map.
\end{proof}

By Proposition~\ref{prop:rigidlattestype} only Latt\`{e}s-type
maps whose orbifolds have signature $(2,2,2,2)$ give genuinely
new maps beyond Latt\`es maps.  We summarize some facts about
these maps in the following discussion. For specific maps see
Examples~\ref{ex:non-expanding-lattes} and~\ref{ex:notattained}.
 
\begin{ex}[Latt\`{e}s-type maps with signature
  $(2,2,2,2)$]
  \label{ex:lattes_type}
  We know that each such Latt\`{e}s-type map arises from a
  crystallographic group $G$ of type $(2222)$ (see
  Corollary~\ref{cor:orbofLattty}). In this case, it is elementary
  to check that the isometries in $G$ remain isometries on $\R^2$
  not only if we conjugate them by an isometry on $\R^2$, but
  even if we conjugate them by an affine homeomorphism
  $h\colon \R^2\to \R^2$. It follows that the class of
  crystallographic groups $G$ of type $(2222)$ is preserved under
  conjugation by such a homeomorphism $h$.  Similarly, the class
  of orientation-preserving affine homeomorphisms
  $A\colon \R^2\to \R^2$ is preserved under conjugation by
  $h$. By the discussion preceding \eqref{eq:LattesGnotunique},
  we may therefore assume in the construction of Latt\`es-type
  maps with orbifold signature $(2,2,2,2)$ that the underlying
  lattice $\Gamma$ of $G$ is equal to the integer lattice
  $\Gamma=\Z^2$ and that the group $G$ consists of all isometries
  $g\: \R^2 \ra \R^2$ of the form
  \begin{equation}
    \label{eq:specisom} 
    u\in \R^2\mapsto g(u)=\pm u+\ga, 
    \quad 
    \text{where  $\ga \in \Gamma=\Z^2$.}
  \end{equation} 

  The quotient space $\R^2/G$ is a $2$-sphere $S^2$. Indeed,  one can
  identify $\R^2/G$ with a pillow $\Delta$ (see Section~\ref{sec:expratThmaps}) in the following way  (the ensuing discussion is closely related to the  more general considerations in Section~\ref{sec:paraorbcov}). Let
$$R\coloneqq [0,1]\times [0,1/2], $$ 
$$S\coloneqq [0,1/2]\times [0,1/2], \quad \text{and} \quad
S'\coloneqq [1/2,1]\times [0,1/2]. $$ Then $R=S\cup S'$ is a fundamental
domain (see Section~\ref{sec:appquotmaps}) for the action of $G$ on $\R^2$, i.e., $R$ contains a
representative from every orbit, and this representative is unique in
$R$ if it lies in the interior of $R$. So $\R^2/G$ is obtained from
$R$ by identifying certain points on the boundary of $R$. For this one
folds the rectangle $R$ along the line $\ell=\{(x,y)\in \R^2: x=1/2\}$
and identifies corresponding points on the boundaries of the two
squares $S$ and $S'$ under this folding operation; so for example the point
$(0, t)\in\R^2 $ is identified with $(1,t)\in \R^2$ for $t\in [0,1/2]$. The pillow $\Delta$
 obtained in this way is our quotient space $\R^2/G$.
 
 Let $\Theta$ be the map that sends the square $S\sub \R^2$ by the identity map
 to $S\sub \Delta$. This maps extends by successive reflections
 in a natural way  to a continuous  map $\Theta\: \R^2 \ra \Delta$. If $g\in G$, then   
  $\Theta$  maps $g(S)$ by an isometry  to one side of $\Delta$, and $g(S')$
 to the other side of $\Delta$. Note that $\Theta$ is the same map as in Section~\ref{sec:Lattes} and corresponds to the quotient map $\R^2\ra \R^2/G$  under the identification $ \Delta\cong  \R^2/G$.

The map $\Theta$ is induced by $G$, and from the geometric description one easily sees that  $\Theta\: \R^2\ra  \Delta\cong \R^2/G$ is a branched covering map (this also follows from Proposition~\ref{prop:G_yields_T}). Its critical points are the corners of the squares 
 $g(S)$ and $g(S')$, $g\in G$, i.e., the points in 
 $\frac12 \Z^2$. We have  $\deg(\Theta,u)=\#G_u=2$
 for $u\in \frac12 \Z^2$ and $\deg(\Theta,u)=\#G_u=1$ for $u\in \R^2 \setminus \frac12 \Z^2$. Note that the set 
 $\Theta(\frac12 \Z^2)$ consists precisely of the four corners of $\Delta$. We define a ramification function $\alpha$ on $\Delta$ that assigns the value $2$ to each of these corners, and $1$ to all other points of $\Delta$. In this way we obtain 
an orbifold   $(\Delta, \alpha)$  with signature $(2,2,2,2)$. If we use some conformal identification $\Delta \cong \CDach$ (as discussed in Section~\ref{sec:Lattes}), then  
 $\Theta\: \R^2\cong \C  \ra \Delta \cong\CDach $ is the (holomorphic)  universal orbifold covering map of this orbifold (see Theorem~\ref{thm:orbunifparacas}). 

Now let $A(u)=L_A(u)+u_0$, $u\in \R^2$,  be an 
orientation-preserving affine homeomorphism. Then $A$ induces a map on the
quotient $\R^2/G$ if and only if $A$ is $G$-equivariant, or equivalently 
if $A\circ g\circ A^{-1}\in G$ for each $g\in G$ (see Lemma~\ref{lem:f_groupdescend}). 
Exactly as in Proposition~\ref{prop:alph_beta_G}, this is the
case if and only if $2u_0\in \Gamma$ and  $L_A(\Gamma)\sub
\Gamma$.  

Since $\Gamma=\Z^2$ the latter condition is true precisely if the matrix representing $L_A$ with
respect to the standard basis in $\R^2$ has integer coefficients.  So we
conclude that the homeomorphism $A$ induces a map on $S^2=\R^2/G$ precisely if $A$ has
the form
\begin{equation}\label{eq:Lattestype2222} 
  A (u)  
= \left(\begin{array}{cc} a &b\\ c& d  \end{array} \right)\left(\begin{array}{c} x\\ y  \end{array} \right) +\frac 12 
\left(\begin{array}{c} x_0\\ y_0  \end{array} \right)  \text{ for } u= \left(\begin{array}{c} x\\ y  \end{array} \right)
\in \R^2, 
\end{equation} 
where $a,b,c,d,x_0,y_0\in \Z$ and $\det(L_A)=ad-bc\ge 1$.
Here the last inequality follows from our assumption that $A$ is orien\-tation-preserving. 

Let $f\: S^2\ra S^2$ be the branched covering   map 
induced by $A$ as in \eqref{eq:Lattestype2222}.  
If $\det(L_A)=1$, then $A^{-1}$ also is of the form \eqref{eq:Lattestype2222}, and so $f$ has 
a continuous  inverse induced  by $A^{-1}$. In this case $f\: S^2 \ra S^2$ is  a homeomorphism.

If $\det(L_A)\ge 2$, then $\deg(f)=\det(L_A)\ge 2$ by Lemma~\ref{lem:deglatttype}. 
Then $f$  is a  Latt\`es-type map, and we know that in this case   the orbifold of $f$ has   signature $(2,2,2,2)$ (see 
Corollary~\ref{cor:orbofLattty}).

In the construction above, we can also  use some other branched covering map $\Theta\: \R^2\ra S^2$ induced by $G$. 
As follows from the uniqueness part in Proposition~\ref{prop:G_yields_T}, this  leads to the same class of Latt\`es-type maps up to topological conjugacy.  
\end{ex}

By  the previous discussion we established   the following statement.
\begin{prop}\label{prop:2222} Let $G$ be the group
consisting  of all iso\-me\-tries  $g\: \R^2 \ra \R^2$ of the form 
\eqref{eq:specisom}, 
and 
$A\:\R^2\ra \R^2$  be an affine orien\-tation-preserving homeomorphism    as in 
\eqref{eq:Lattestype2222}.

Then $A$ descends to a map $f\: S^2\ra S^2$ on the quotient $\R^2/G \cong S^2$, i.e., if $\Theta\: \R^2\ra \R^2/G\cong S^2$ is the quotient map, then $\Theta\circ A=f\circ \Theta$.
 
 If $ad-bc=1$, then $f$ is a homeo\-morphism that is orientation-preserving. If $ad-bc\ge 2$, then $f$ is a 
 Latt\`es-type map whose orbifold has signature $(2,2,2,2)$. 

Moreover, every  Latt\`es-type map with  orbifold  signature $(2,2,2,2)$ is topologically conjugate to a map $f$ obtained in this way. 
\end{prop}

An obvious question is when a Latt\`es-type map $f$ is Thurston equivalent to a rational map. This is always true if the signature of $\mathcal{O}_f$  is equal to   $(2,4,4)$, $(3,3,3)$, or $(2,3,6)$, because then $f$ is even conjugate to a Latt\`es map (Proposition~\ref{prop:rigidlattestype}). 

If  $\mathcal{O}_f$  has signature $(2,2,2,2)$, the question is
answered by the following statement which can be seen as
complementary to Thurston's characterization of rational Thurston maps with hyperbolic orbifold
(Theorem~\ref{thm:Thurston}). 

\begin{theorem}
[Rationality of Latt\`es-type maps]
  \label{thm:Thurston_para}
  \index{Thurston!obstruction}
  Let $f\colon S^2\to S^2$ be a Latt\`{e}s-type map with orbifold signature
  $(2,2,2,2)$ and  $A\colon
  \R^2\to \R^2$ be an affine map as in Definition~\ref{def:Lattestype}  with linear part $L_A$. Then $f$ is Thurston
  equivalent to a rational map if and only if $L_A$ is a real multiple of
  the identity map on $\R^2$  or the eigenvalues of $L_A$ belong to  $\C\setminus \R$. 
\end{theorem}
We will not prove this  here, but refer to 
\cite[Proposition~9.7]{DH} for an essentially equivalent statement. 

In the final part of  this section we provide the  proof of Proposition~\ref{prop:noperparaLTM} that 
characterizes  Latt\`es-type maps up to 
Thurston equivalence.

  \begin{proof}[Proof of Proposition~\ref{prop:noperparaLTM}.] 
  Suppose first that $f$ is a Thurston map that is Thurston equivalent to 
  a   Latt\`es-type map $g$. Then $f$ and $g$ have the same orbifold signature
  (Proposition~\ref{prop:Thequivsamesig}). Since $g$ has a parabolic orbifold and no periodic critical points by Proposition~\ref{prop:immediate}, the same is true for the map $f$ as follows from Propositions~\ref{prop:parabolicOf} and ~\ref{prop:otherramprops}~\ref{item:rami_infty}.

  For the converse direction, suppose that $f$ is a Thurston map with parabolic orbifold $\mathcal{O}_f$ and no periodic critical points. We know that then the signature of 
  $\mathcal{O}_f$ is $(2,2,2,2)$, $(2,4,4)$, $(2,3,6)$, or $(3,3,3)$. 
  
  In the last three cases the map has precisely three 
  postcritical points. As we will see later (Theorem~\ref{thm:3postrat}), every 
  Thurston map $f$ with three postcritical points is Thurston equivalent to a rational map $R$. Then the signatures 
  of the orbifolds of $f$ and $R$ are the same 
  (Proposition~\ref{prop:Thequivsamesig}), and so $R$ is a  Thurston map with a parabolic orbifold $\mathcal{O}_R$ and no periodic critical points. By 
  Definition~\ref{def:Lattes} the map $R$ is a Latt\`es map, and the statement follows in this case.   
   
  So we are left with the case where  $\mathcal{O}_f$ has signature $(2,2,2,2)$. We may assume that $f$ is defined on $\CDach$. Let $\alpha=\alpha_f$ be the ramification function of $f$,  and $\Theta\: \C\cong \R^2\ra  \CDach$ be the holomorphic  universal orbifold covering map of $\mathcal{O}_f=(\CDach,\alpha)$ as provided by Theorem~\ref{thm:orbunifparacas}.  
  The group $G$ of deck transformations of $\Theta$ is a
  crystallographic group of  type $(2222)$.  
  
 Replacing $\Theta$ with $\Theta\circ h$  for suitable $h\in \Aut(\C)$
 if necessary (which changes the deck transformation group to $h^{-1}\circ G \circ h$), we may assume that $G$ is equal to a group $\widetilde G$ of type $(2222)$ as  in Theorem~\ref{thm:G_signature}. If  $\Gamma$ is  the underlying lattice of $G=\widetilde G$, then 
  $G$ consists precisely of the isometries on $\R^2$ of the form 
  $u\mapsto \pm u+\ga$, $\ga\in \Gamma$.  We also consider the torus 
   $T^2=\R^2/\Gtr=\R^2/\Gamma$ and the quotient map $\pi: \R^2\ra \R^2/\Gamma$.
  
   We now repeat part of the arguments for the proof of the  
  implications  \ref{item:Lattessrucii} 
  $\Rightarrow$ \ref{item:Lattessruciii} and \ref{item:Lattessruciii} 
  $\Rightarrow$ \ref{item:Lattesruciv}  in 
  Theorem~\ref{thm:Lattesstruc}.
 The difference is that we do not have holomorphicity of the maps involved, but this  is mostly inessential.  First, the parabolicity of $\mathcal{O}_f$ in combination with the uniqueness part of Theorem~\ref{thm:orbunifparacas} implies that there exists an orientation-preserving  homeomorphism $A\: \R^2\ra \R^2$ such that $f\circ \Theta =\Theta \circ A$. 
Again $A$ is $G$ and $\Gtr$-equivariant. This allows us to push the map 
$A$ to $T^2$, and we obtain 
 a diagram as in \eqref{eq:Athetafpi34} (with $S^2=\CDach$), where 
 $\overline \Theta\: T^2\ra \CDach$ is a branched covering map and $\overline A\: T^2\ra T^2$ is a torus endomorphism. 
 
 In particular, $f$ is a quotient of the torus endomorphism $\overline A$ 
 and so by Lemma~\ref{lem:simobserv}~\ref{item:simobii} the set $\post(f)$ is equal to the set of 
 critical values of $\overline \Theta$. Since $\Theta =\overline \Theta\circ\pi$ and $\pi$ is a covering map, the set of critical values of  $\Theta$ and 
$ \overline \Theta$ are the same. Now 
$$ \crit(\Theta)=\{ u\in \R^2: \#G_u=2\}=\tfrac12 \Gamma$$ and it follows that 
\begin{equation} \label {eq:postfGa}
\post(f)=\Theta\big(\tfrac12 \Gamma\big).
\end{equation} 

By Lemma~\ref{lem:torilifts}~\ref{item:tori3} there exists a linear map   $L\: \R^2\ra \R^2$ (essentially 
the map induced by $\overline A$ on the fundamental group $\Gamma$ of $T^2$)  with $L(\Gamma)\sub \Gamma$ 
such that 
\begin{equation} \label{eq:AtauLrel} 
A\circ \tau_\ga \circ A^{-1} =\tau_{L(\ga)}
\end{equation}
for $\ga\in \Gamma$, where $\tau_\gamma$ denotes the translation $u\in \R^2\mapsto \tau_\ga(u)\coloneqq u+\ga$. This  relation can be rewritten as
\begin{equation} \label{eq:AtauLrel2} 
A(u+\ga)=A(u)+L(\ga) \quad\text{for each $u\in \R^2$, $\ga \in \Gamma$.} 
\end{equation}

By  Lemma~\ref{lem:torilifts}~\ref{item:tori4} we have $\det(L)=\deg(\overline A)=\deg(f)\ge 2$, and so $L$ is an 
orientation-preserving linear homeomorphism on $\R^2$. 

The $G$-equivariance of $A$ also implies that for suitable sign and $\ga_0\in \Gamma$ we have
$ A(-u)=\pm A(u)+\gamma_0$ for each $u\in \R^2$. Here we necessarily have the negative sign on the right hand side; otherwise, by setting $u=0$ we obtain $\ga_0=0$ and $A(u)=A(-u)$ for $u\in \R^2$. This is impossible, because
$A$ is a homeomorphism and hence injective. 

It follows that $ A(-u)= -A(u)+\gamma_0$ for $u\in \R^2$;  setting $u=0$ we see that 
$\gamma _0= 2A(0)$. 
We conclude that 
\begin{equation} \label{eq:AtauLrel3} 
A(0)\in \tfrac12 \Gamma 
\end{equation}
and 
\begin{equation} \label{eq:AtauLrel4}
A(-u)=-A(u)+2A(0) \text{ for each $u\in \R^2$.} 
\end{equation}

Now we define 
$$ \widetilde A(u)=L(u)+A(0)$$ 
for $u\in \R^2$. Then  $ \widetilde A$ is an orientation-preserving affine 
homeomorphism whose linear part $L$ satisfies 
$\det(L)\ge 2$. Moreover, the map  $\widetilde A$ satisfies  relations as in 
\eqref{eq:AtauLrel2} and \eqref{eq:AtauLrel4}. This together with the fact 
$\widetilde A(0)=A(0)\in \tfrac12 \Gamma $ implies that $ \widetilde A$ is 
$G$-equivariant. It follows that there exists a Latt\`es-type map $g\: \CDach 
\ra \CDach$ with $g\circ \Theta= \Theta\circ \widetilde A$. 

We claim that $f$ and $g$ are Thurston equivalent. To see this, consider 
the orientation-preserving homeomorphism $B\coloneqq \widetilde A^{-1}\circ A$
on $\R^2$.
Note that $\widetilde A^{-1}(u)=L^{-1}(u-A(0))$ for $u\in \R^2$. 
This and the relations \eqref{eq:AtauLrel2} and \eqref{eq:AtauLrel4} imply that
\begin{equation}\label{eq:Beqvar} 
B(\pm u+\ga)=\pm B(u)+\ga \text{ for each $u\in \R^2$, $\ga \in \Gamma$.} 
\end{equation}
  In particular, $B$ is $G$-equivariant and so there exists 
a continuous map $h\: \CDach\ra \CDach$ such that 
\begin{equation}\label{eq:hBrel}
h\circ \Theta =\Theta\circ B.
\end{equation} 

Note that 
$$ g\circ h\circ \Theta= g \circ \Theta\circ B= \Theta \circ \widetilde A \circ B=
\Theta\circ A= f\circ \Theta. $$
Since $\Theta\: \R^2\ra \CDach$ is surjective, we conclude that 
$ g\circ h=f. $
So the Thurston equivalence of $f$ and $g$ will follow if we can show that 
$h$ is a homeomorphism that is isotopic to  $\id_{\CDach}$ rel.~$\post(f)$.

It is easy to see that the inverse map $B^{-1}$ satisfies a similar relation as in \eqref{eq:Beqvar}. Hence $B^{-1}$ is also $G$-equivariant and descends to a continuous map on $\CDach$. This map is an inverse map for $h$, and it follows that $h$ is a homeomorphism. Since $B$ is orientation-preserving, the same is true for $h$ (this easily follows from 
\eqref{eq:hBrel} and Lemma~\ref{lem:23orient}). 

Using  \eqref{eq:Beqvar} with the negative sign and $u=\frac12 \gamma$, 
one sees that   
$$ B(\tfrac 12 \ga)=\tfrac 12\ga \text{ for each $\ga \in \Gamma$.} $$ 
If we combine this  with \eqref{eq:postfGa} and \eqref{eq:hBrel}, we conclude  that $h$ fixes each of the four  points in $\post(f)$. 

It remains to show that $h$ is isotopic to $\id_{\CDach}$ rel.~$\post(f)$. There is no easy self-contained argument for this and we have to invoke 
some facts about the mapping class group  of a sphere with four punctures 
(as given by the points in $\post(f)$) and its relation to the mapping class group of a torus (for the definition of the mapping class group and the facts needed see  
 \cite[Sections~2.1~and~2.2]{FM}). 

We have a well-defined torus involution $\overline I\: T^2 \ra T^2$ 
induced by the map $u\in \R^2\mapsto I(u)=-u$. Then $\pi\circ I=\overline I\circ \pi$ and $\overline I$ is the generator of a cyclic group $\overline G$ 
of order $2$ acting on $T^2$. The map $\overline \Theta\: T^2 \ra \CDach$ corresponds to the 
quotient map $T^2\ra T^2/\overline G\cong \CDach$. Moreover, $\overline I$ has 
four fixed points given by the images of the points in $\frac 12 \Gamma$ 
under the quotient map $\pi\: \R^2\ra T^2\cong\R^2/\Gamma$. 
These fixed points of $\overline I$ in turn are mapped by $\overline \Theta$ to the points in the set $P\coloneqq \post(f)$ by $\overline \Theta$ as follows from \eqref{eq:postfGa}.

It is now a general fact that an orientation-preserving 
 homeomorphism on $\CDach $ fixing the points in  
$P$ (i.e., the images of the fixed points of  $\overline I$ under the projection map $\overline \Theta$) is isotopic  to the identity
rel.~$P$  if it has a lift to $T^2$ by $\overline \Theta$ that induces the identity map on the fundamental group of $T^2$ (this is essentially shown in \cite[Proof of Proposition~2.7, pp.~59--60]{FM}).

 In our situation, 
such a lift to $T^2$ of the homeomorphism $h$ can easily be obtained. Namely, 
it follows from  \eqref{eq:Beqvar} 
that there exists a homeomorphism 
$\overline B\: 
T^2\ra T^2$ 
with
$\pi\circ B=\overline B\circ \pi$. Then 
 by definition  $B$ is a lift of $\overline B$ by $\pi$. Now
 the translations $u\mapsto u+\ga$, $\gamma\in \Gamma$, form the group of deck transformations  of $\pi$ which represents the fundamental group of $T^2$. Then  it follows from \eqref{eq:Beqvar} that the induced map of $\overline B$ on the fundamental group is the identity (see the remarks after Lemma~\ref{lem:torilifts} for a related discussion). Finally, $\overline B$ is a lift of $h$ 
 by $\overline \Theta$, because we have 
 $$ \overline \Theta\circ \overline B\circ \pi =  \overline \Theta\circ \pi \circ B 
 =\Theta \circ B=h\circ \Theta= h\circ\overline  \Theta\circ \pi, $$
 and so 
 $$ \overline \Theta\circ \overline B=h\circ\overline  \Theta. $$
 This shows  that 
 $h$ is indeed isotopic to $\id_{\CDach}$ rel.~$P=\post(f)$.
  \end{proof}

\section{Covers of parabolic orbifolds}
\label{sec:paraorbcov}

In this section we provide  the proofs of 
Proposition~\ref{prop:G_yields_T} and Theorem~\ref{thm:orbunifparacas}.  The main point   is to prove existence of the map $\Theta$ 
 in these statements.   We will do this in an explicit geometric way that is also useful for visualizing  examples of 
Latt\`es maps (see Section~\ref{sec:examples-lattes-maps}).

We start with  a given crystallographic group $G$ 
not isomorphic to $\Z^2$.
We first consider the types $(244)$, $(333)$, and $(236)$. The type $(2222)$  is different  and will be treated  later. So let  $G$ be of type (244). Our goal is to find a holomorphic map 
$\Theta\: \C \ra \CDach$ induced by $G$.  We explain the construction in detail only  in this case. For the types 
 $(333)$ and $(236)$ the considerations are 
completely analogous and we will skip the details.

The group   $G$ has an invariant tiling  made out of 
isometric  copies of a 
right-angled Euclidean triangle with angles $\pi/2$, $\pi/4$, $\pi/4$ as shown in Figure~\ref{fig:244_univ_orb}. The triangles are  colored black or white. The 
union of a white and  a black triangle with a common edge forms
a fundamental domain for the action of $G$ on $\C$. 
Let $T$ be one of the white triangles in this tiling.   We
glue an isometric copy $T_\wt$ of $T$, colored white, with another
isometric copy $T_\bt$ of $T$, colored black, along their boundaries
to form a pillow $\Delta$ (see Section~\ref{sec:expratThmaps}). 
 Then $\Delta$ is a topological
$2$-sphere and can be identified with  the quotient space $\C/G$.  
The identification $T\cong T_w$ induces an  orientation on  $\Delta$ 
(see Section~\ref{sec:expratThmaps}).
 We  equip
$\Delta$ with the unique path metric   that  restricts to the Euclidean metric
on the two copies of $T$ (see Figures~\ref{fig:lattes244a},
~\ref{fig:Lattes333}, and~\ref{fig:lattes236}
for an illustration of $\Delta$ for the different types of $G$ considered).

 One can now define 
a map $\Theta_\Delta\colon \C \to \Delta$
 as follows. The map $\Theta_\Delta$ sends each
white triangle $T\subset \C$ from the tiling as represented in 
Figure~\ref{fig:244_univ_orb} to (the white triangle) $T_{\wt}\subset \Delta$  and each black triangle $T\subset \Delta$ to (the black triangle) $T_{\bt}\subset
\Delta$ by an orientation-preserving isometry. 
   Then  
$\Theta_\Delta\colon \C \to \Delta$ is a well-defined continuous map.
  
Note that if $G$ is of type $(333)$, then a similar  construction does not lead to a well-defined map due to a rotational ambiguity which is not present for the types $(244)$ and $(236)$. In this case, one has to single out  
one of the common vertices $v$ of $T_\wt$ and $T_\bt$ and impose the additional requirement  that the piecewise isometry $\Theta_\Delta$ sends the points marked by a black dot in Figure~\ref{fig:333_univ_orb} to $v$.   

It is clear from the definition 
of $\Theta_\Delta$ that $\Theta_\Delta=\Theta_\Delta\circ g$ for each $g\in G$. On the other hand, if $\Theta_\Delta(z)=\Theta_\Delta(w)$ for $z,w\in \C$, then we may pick two triangles $T$ and $T'$ of  the same color in the tiling as represented by Figure~\ref{fig:244_univ_orb} such that $z\in T$ and $w\in T'$.
There is a unique element $g\in G$ with $g(T)=T'$. Then $g(z),w\in T'$ and 
$$\Theta_\Delta(g(z))=\Theta_\Delta(z)=\Theta_\Delta(w). $$
Since $\Theta_\Delta$ is an isometry on $T'$ and hence injective on $T'$, it follows that $w=g(z)$. This shows that $\Theta_\Delta(w)=\Theta_\Delta(z)$ for $z,w\in \C$ if and only if there exists $g\in G$ such that $w=g(z)$. Hence $\Theta_\Delta$ is induced by $G$.

The  $2$-sphere $\Delta$ is a polyhedral surface equipped with a locally Euclidean metric with three conical
singularities (as shown in Figure~\ref{fig:lattes244a}). In particular, $\Delta$ carries a natural conformal structure,
with respect to which it is conformally equivalent to $\CDach$. Moreover, 
$\Theta_\Delta\: \C \ra \Delta$ is a continuous map that is a local isometry near each point 
in $\C$ that is not a  preimage of one of the three cone points of $\Delta$. Hence 
$\Theta_\Delta\: \C \ra \Delta$ is a holomorphic map (all this is explained in greater generality in Section~\ref{sec:expratThmaps}). 

It follows from the uniformization theorem that  we  can find   a conformal
map $\varphi \colon \Delta \to \CDach$ that sends  the vertices of the pillow
to the points $0$, $1$, $\infty$ in $\CDach$. In fact, one can construct $\varphi$ 
quite explicitly by first  mapping $T_{\wt}$
conformally to the upper half-plane, and $T_{\bt}$  to the
lower half-plane such that the vertices of the triangles are sent to $0$, $1$, $\infty$.
If we define
 $\Theta = \varphi \circ\Theta_\Delta \colon \C\to
\CDach$, then  $\Theta$ is a holomorphic map 
 induced by $G$. 
 
 We will verify the other properties of $\Theta$ as specified in
 Proposition~\ref{prop:G_yields_T} later in this section  by a
 general argument. It  applies to  all types of $G$ and does not
 use our specific construction. 
It is still illuminating   
to see   how
 these properties can be extracted from the tiling in  Figure~\ref{fig:244_univ_orb}, at least on an intuitive level. First, it is clear that $\Theta_\Delta$, and hence also 
   $\Theta$, is a branched covering map: if $q\in \Delta$ is arbitrary, then an   open ball $V\sub \Delta$ centered at $q$ of small enough radius $\eps>0$    is evenly covered by $\Theta_\Delta$ (see Definition~\ref{def:brcovmap}). Each component $U$ of  $\Theta_\Delta^{-1}(V)$ is a Euclidean disk of radius 
   $\eps>0$ centered at a point  $z\in \Theta^{-1}_\Delta(q)$. Each 
   point $q'\in V$ with $q'\ne q$ has  precisely $d$ distinct  preimages  in $U$, where 
   $d=\#G_z$. So $\Theta_\Delta$ is $d$-to-$1$ near $z$. 
   Hence $\deg(\Theta_\Delta,z)=\deg(\Theta,z)=\#G_z$ for $z\in \C$. Since $\#G_z$ is the same 
   for each point $z$ in a given $G$-orbit and 
   $\Theta_\Delta$ is induced by $G$, we can define a ramification function 
   $\alpha_\Delta\: \Delta \ra \N$ by setting $\alpha_\Delta(p)=\#G_z$ for $p\in \Delta$, where 
   we pick any point $z\in \Theta^{-1}_\Delta(p)=Gz$.  Then   $\alpha_\Delta(p)=2$ if $p$ is the corner of $\Delta$ corresponding to the common vertex of $T_{\tt w}$ and $T_{\tt b}$ with angle $\pi/2$, $\alpha_\Delta(p)=4$ for the other two corners $p$ of $\Delta$, and 
   $\alpha_\Delta(p)=1$ for all other points $p\in \Delta$. So if we set $\alpha\coloneqq 
   \alpha_\Delta\circ \varphi^{-1}$, then $\alpha$ is a finite ramification function on $\CDach$  satisfying 
   \eqref{eq:degT_orderG}.   The orbifold $(\CDach,\alpha)$ is parabolic and has conical singularities at 
 $0$, $1$, $\infty$. Its signature is $(2,4,4)$ corresponding to the type of $G$. 
 
By \eqref{eq:degT_orderG} the holomorphic branched covering  map $\Theta\: \C \ra \CDach$ is the universal orbifold covering map  of $(\CDach, \alpha)$.  As was briefly discussed in
Section~\ref{sec:orbif-assoc-thurst}, we can push forward the Euclidean metric $d_0$ by $\Theta$ and obtain  the 
canonical  orbifold  metric 
$\om$\index{canonical orbifold!metric}\index{metric!canonical orbifold}\index{o@$\omega$}\index{orbifold!canonical metric}  
on $\CDach$ (see Section~\ref{sec:expratThmaps}  for more details). If we equip  $\CDach$ with this metric, then 
$\varphi$ is in fact an isometry.  So $(\CDach, \om)$ and the pillow 
$\Delta$ are isometric, and one can view $\Delta$ as a geometric realization of the orbifold $(\C, \alpha)$.   In particular,   $(\CDach, \om)$
is locally isometric to $\C$, except at the conical singularities 
of the orbifold $(\CDach, \alpha)$ (i.e., the points $p\in \CDach$ with $\alpha(p)\geq 2$),  where
$(\CDach,\omega)$ is locally isometric to a Euclidean cone of
angle $2\pi/\alpha(p)$.

If $(\CDach, \widetilde \alpha)$ is an arbitrary orbifold with signature $(2,4,4)$, then we can 
find a M\"obius transformation $\psi \: \CDach \ra \CDach$ that matches up the three cone points of $(\CDach, \widetilde \alpha)$ and $(\CDach,\alpha)$ such that $\widetilde \alpha 
\circ  \psi
= \alpha$. Then $\psi\circ \Theta\: \C \ra \Cdach$ is the (holomorphic) universal orbifold covering  map of $(\CDach, \widetilde \alpha)$. The geometric pictures remains the same: if we equip    $(\CDach, \widetilde \alpha)$ with its (possibly rescaled) universal orbifold metric $\widetilde \om$, then 
$(\CDach, \widetilde \om)$ is isometric to the pillow $\Delta$.

We now turn to crystallographic groups $G$ of type $(2222)$. In order to construct a holomorphic map $\Theta\: \C \ra \CDach$ induced by
$G$, we may assume that   $G$  consists of all isometries on $\C$ of the form $z\mapsto \pm z+\ga$, where $\ga \in \Gamma$. Here $\Gamma$ 
is  the  underlying  (rank-$2$) lattice $\Gamma\sub \C$ of $G$. 
 For  a general  group $G$ of type $(2222)$  one considers $\Theta\circ h$ with suitable $h\in \Aut(\C)$ to obtain a map induced by $G$ (this reduction is based on Theorem~\ref{thm:G_signature}
 and related to  the remarks at the beginning of Section~\ref{sec:clas-latt-maps}). 
  
Holomorphic maps  $\Theta\: \C \ra \CDach$  
induced by a group $G$ of this special form can be obtained from the 
Weierstra\ss\ $\wp$-function (see \cite[Section~7.3]{A} for
general background).
Recall that  the Weierstra\ss\ $\wp$-function for a given lattice  $\Gamma\sub \C$ 
is defined as 
\index{Weierstrass p-function@Weierstra\ss\ $\wp$-function}
\index{p@$\wp$}
\begin{equation}\label{eq:defwp} 
 \wp(z; \Gamma)=\frac{1}{z^2}+\sum_{\ga\in \Gamma\setminus\{0\}}
\biggl(\frac {1}{(z-\ga)^2}-\frac {1}{\ga^2}\biggr) \quad \text{for $z\in \C$}.
\end{equation} 
Then $\wp=\wp(\cdot\,; \Gamma)$ is an even  meromorphic function on $\C$
with the period lattice $\Gamma$. 
The function  $\wp$ satisfies  the differential equation 
\begin{equation}\label{eq:wpdiff} 
 (\wp')^2=4(\wp-e_1)(\wp-e_2)(\wp-e_3). 
 \end{equation} 
The numbers  $e_1,e_3,e_3\in \C$ here are  three distinct values (depending on $\Gamma$) with 
\begin{equation}\label{eq:sume=0}
e_1+e_2+e_3=0.
\end{equation} 
The critical values of 
 $\wp\: \C \ra \CDach$ are $e_1$, $e_2$, $e_3$, $\infty$. 
 Actually, $\wp\: \C \ra \CDach$  is a holomorphic   map satisfying 
\begin{equation}\label{eq:wpdeg}
\deg(\wp, z)=
\left\{ \begin{array} {cl} 2& \text{if $\wp(z)\in \{e_1, e_2, e_3,\infty\} $,} \
\\ 1 & \text{otherwise.}  \end{array}
 \right.
\end{equation}

If $G$ is the crystallographic group corresponding to $\Gamma$ as above, then $\wp$ is induced by $G$. Indeed, $\wp(w)=\wp(z)$ if and only if 
$w=\pm z+\ga $ for some $\ga\in \Gamma$
(the ``only if" implication  can be derived from the familiar  fact that $\wp$ descends to a holomorphic map 
$\overline \wp\: \T\ra \CDach$  on the torus $\T=\C/\Gamma$ with $\deg(\overline \wp)=2$).  

As is well known and classical, one can reverse this procedure
and start with the differential equation \eqref{eq:wpdiff}: if
three distinct values $e_1$, $e_2$, $e_3\in\C$ with
\eqref{eq:sume=0} are given, then there exists a (unique)
lattice $\Gamma$ such that the corresponding function
$\wp=\wp(\cdot\,; \Gamma)$ satisfies \eqref{eq:wpdiff}. See
\cite[Sections~7.3.3 and 7.3.4]{A}.

The general argument in the proof of Proposition~\ref{prop:G_yields_T} will show that $\wp\: \C \ra \CDach$ is a  branched covering map, because $\wp$  is induced by $G$. It then  follows from  \eqref{eq:wpdeg} 
that $\wp$ is the universal orbifold covering map 
of the orbifold $(\CDach, \alpha)$, where 
$\alpha(p)=2$ for  $p\in  \{e_1, e_2, e_3,\infty\}$ and 
$\alpha(p)=1$ for $p\in \CDach \setminus \{e_1, e_2, e_3,\infty\}$. 
This implies that the universal orbifold covering map $\Theta$
of any orbifold  $(\CDach, \alpha)$ with signature $(2,2,2,2)$ can always be obtained from a Weierstra\ss\ $\wp$-function followed by a suitable M\"obius transformation. Indeed, if $p_1, \dots, p_4\in \CDach$  are the four distinct points with $\alpha(p_k)=2$ for $k=1, \dots, 4$, then there exists a M\"obius transformation $\psi$ on $\CDach$ such that $\psi(\{p_1, \dots, p_4\})=\{e_1, e_2, e_3, \infty\}$, 
where $e_1, e_2, e_3\in \C$ are three distinct points satisfying    \eqref{eq:sume=0}
(first map $p_4$ to $\infty$ by a M\"obius transformation and then apply a suitable translation). 
Then we can find  a lattice $\Gamma$ such that the corresponding 
function $\wp=\wp(\cdot\,; \Gamma)$ satisfies \eqref{eq:wpdiff} with these values $e_1, e_2,e_3$. If we define $\Theta= \psi^{-1}\circ \wp(\cdot\,; \Gamma)$, then $\Theta$  is the universal 
orbifold covering map of the given orbifold   $(\CDach, \alpha)$.  

Actually, one can make one more reduction here. Namely, in the last statement 
we  may assume that the lattice  has the form $\Gamma =\Z \oplus\Z\tau$  with $\tau \in \C$ and $\imag(\tau)>0$. This  follows  from the homogeneity property 
$$\wp(\lambda z; \lambda \Gamma)= \frac1{\lambda^2} \wp(z; \Gamma), \quad z\in \C, \, \lambda\in \C\setminus\{0\}, 
$$ 
of the $\wp$-function.

We will now describe a more   geometric construction for the maps $\Theta$ and their associated orbifolds that arise here     similar  to the one given for the crystallographic groups of type 
(244).  For this we fix $\tau \in \C$ with $\imag(\tau)>0$, and consider the crystallographic group $G$ of type $(2222)$ given by all isometries $z\mapsto \pm z+\ga$, where $\ga \in \Gamma\coloneqq \Z\oplus \Z\tau$. Not all crystallographic groups of type $(2222)$ have this form, but we restrict ourselves to the special case for simplicity.   One can easily 
adjust the ensuing discussion to the general case by
essentially precomposing all relevant maps on $\C$ with  a suitable element $h\in \Aut(\C)$.  
 Note that by the reduction discussed above we still get the same class of orbifolds from this restricted class of groups. 

A fundamental domain for $G$ is 
 the Euclidean triangle $T \subset \C$ with vertices $0,1, \tau$. To obtain the quotient space $\C/G$ from $T$, 
we divide $T$ into four similar triangles by connecting the midpoints of
each side and  fold $T$ along the edges of these smaller triangles. In this way,
we can build a  tetrahedron $\Delta$ in Euclidean $3$-space as
indicated in Figure~\ref{fig:tetrahedron} (the two halves of each side
are identified). Here we allow the degenerate case that $\Delta$ is a
pillow. This happens precisely when $T$ has a right angle. 

 We equip  $\Delta$ with the path metric induced by the
Euclidean metric on $T$, and the orientation 
induced by the orientation on  $T\subset \C$. 
Then $\Delta$ is homeomorphic to a  $2$-sphere. 
It is   a polyhedral surface 
 with four conical singularities and carries a natural conformal structure.

 \begin{figure}
    \centering
    \begin{overpic}
      [width=10cm, tics = 20,
      ]
      {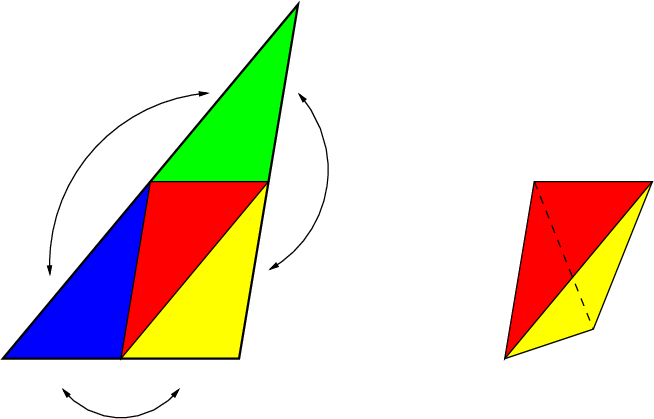}
      \put(-1,5){$0$}
      \put(36,5){$1$}
      \put(47,62.5){$\tau$}
      \put(10,50){$T$}
      \put(90,40){$\Delta$}
    \end{overpic}
    \caption{Folding a tetrahedron from a triangle.}
    \label{fig:tetrahedron}
  \end{figure}

  \begin{figure}
    \centering
    \begin{overpic}
      [width=10cm, tics =20,
      ]
      {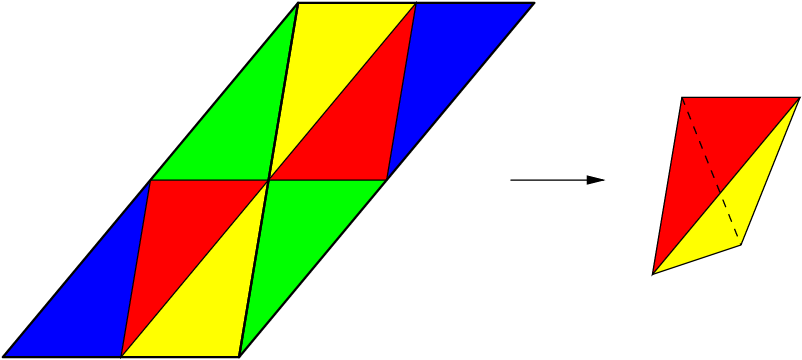}
      \put(13,27){$P$}
      \put(32,44){$\tau$}
      \put(32,-2){$1$}
      \put(-3,-2){$0$}
      \put(90,38){$\Delta$}
    \end{overpic}
    \caption{Construction of $\Theta = \wp$. }
    \label{fig:constr_wp}
  \end{figure}

  The parallelogram $P\subset \C$ spanned by $1$ and $\tau$ is a
  fundamental domain for the group of translations $\Gtr\sub G$ given by all
  maps of the form $z\mapsto z+\ga$, where
  $\ga\in \Gamma=\Z \oplus  \Z \tau$. We can divide $P$ into two
  triangles that are isometric to $T$, and into $8$ smaller triangles
  similar to $T$ by the scaling factor $2$ (see
  Figure~\ref{fig:constr_wp}).

  As indicated in Figure~\ref{fig:constr_wp}, there is a map
  $\Theta_\Delta\: P\to \Delta$ such that each small triangle in
  $P$ is mapped isometrically to a face of $\Delta$. This map extends
  naturally to a continuous map $\Theta_\Delta\colon\C\to \Delta$ that
  respects the lattice translations in the sense that  $\Theta_\Delta \circ
  g= \Theta_\Delta$ for all $g\in \Gtr$. Note that then
  $\Theta_\Delta(z)=\Theta_\Delta(-z)$ for $z\in \C$, and it
  easily follows that $\Theta_\Delta$ is induced by $G$. In particular,  we
  can identify $\C/G$ with the tetrahedron $\Delta$ (this follows from Corollary~\ref{cor:groupquot}~\ref{item:groupquot2}).  

By the  uniformization theorem   we can   find a conformal map   $\varphi\colon \Delta\to \CDach$. 
If $\varphi$ is such a map, then
  $\Theta=\varphi\circ\Theta_\Delta \colon \C\to\CDach$ is a
  holomorphic map  induced by $G$. This (or a direct geometric argument) implies that $\Theta$ and $\Theta_\Delta$ are branched covering maps. If we choose a suitably normalized
  map  $\varphi=\varphi_0$ here, 
  then $\Theta=\varphi_0\circ \Theta_\Delta=\wp(\cdot\,; \Gamma)$.

  If we assign to each of the 
four conical singularities of $\Delta$ (where the cone angle is $\pi)$ the value $2$ for a ramification function $\alpha_\Delta$ on
$\Delta$ and the value $1$ to all other points,  then $(\Delta, \alpha_\Delta)$ is an orbifold with signature $(2,2,2,2)$ and we obtain the relation \eqref{eq:degT_orderG} (for $\alpha_\Delta$ and $\Theta_\Delta$). So if  we define $\alpha=\alpha_\Delta\circ \varphi^{-1}$, then $(\CDach, \alpha)$  
is an orbifold with the same signature as   $(\Delta, \alpha_\Delta)$ and \eqref{eq:degT_orderG} holds. Then 
$\Theta$ is the universal orbifold covering map of $(\CDach, \alpha)$. If we equip 
$\CDach$ with the corresponding universal orbifold metric $\om$, then $(\CDach, \om)$ is isometric to $\Delta$ (possibly up to scaling). \index{canonical orbifold!metric}\index{metric!canonical orbifold}\index{o@$\omega$}\index{orbifold!canonical metric} 
 
  It follows from our earlier analytic discussion that if we set   $\varphi=\psi\circ \varphi_0$ for  a suitable M\"obius transformation $\psi$ and choose  
     $\tau\in \C$ with $\imag(\tau)>0$ appropriately,  then 
     the   universal orbifold covering map $\Theta$ of any orbifold $(\CDach, \alpha)$ with  signature $(2,2,2,2)$ can be written in the form 
    $\Theta=\varphi\circ \Theta_\Delta$.

We can now give the  proofs of Proposition~\ref{prop:G_yields_T} and Theorem~\ref{thm:orbunifparacas}.

\begin{proof}[Proof of Proposition~\ref{prop:G_yields_T}] 
Let  $G$ be  a crystallographic group  not isomorphic to $\Z^2$. Then $G$ is of type 
$(244)$, $(333)$, $(236)$, or $(2222)$. We have seen earlier 
in this section how to construct a holomorphic map $\Theta\: \C \ra \CDach$ induced by $G$   for each of these types. 
For logical clarity we will not use any other facts from our previous  discussion.


Since $G$ acts cocompactly on $\C$ and $\Theta$ is induced  by $G$,
 the image $\Theta(\C)$ is compact. This implies that $\Theta$  is surjective, because 
$\Theta(\C)$ is also open in $\CDach$, and so $\Theta(\C)=\CDach$.

  In order to see  the second identity in \eqref{eq:degT_orderG}, let
  $z_0\in \C$ be arbitrary  and $d\coloneqq \#G_{z_0}$.  We have to show that 
  $\Theta$ is locally $d$-to-$1$ near $z_0$.
  Since $G$ acts properly discontinuously on $\C$, the orbit $Gz_0$ is a discrete set in  
  $\C$. This implies that each isometry  $g\in G\setminus G_{z_0}$ moves points 
  near $z_0$ a definite distance away from $z_0$.  So if we choose the radius $\eps>0$ of the Euclidean disk $B \coloneqq 
B_\C(z_0,\eps)$ sufficiently small, then $B\cap g(B)\ne \emptyset$ for some $g\in G$ if and only if $g\in G_{z_0}$.  It follows that 
\begin{equation}\label{eq:ghBGz0}
\text{$g(B)\cap h(B)\ne \emptyset$ for $g,h\in G$ if and only if $
h^{-1}\circ g\in G_{z_0}$.}
\end{equation}

Now if $u,v\in B$ and $\Theta(u)=\Theta(v)$, then there exists $g\in G$ with $v=g(u)$, because $\Theta$ is induced by $G$. In particular, $v\in B\cap g(B)$, and so 
$g\in G_{z_0}$ by \eqref{eq:ghBGz0}.   We conclude that two points in $B$ have the same image under $\Theta$ if and only if they belong to the same $G_{z_0}$-orbit. Now $G_{z_0}$ is a cyclic group of order $d$ consisting of rotations around $z_0$. Hence the orbit of each point $z\in B\setminus \{z_0\}$ consists of precisely $d$ points in $B$. This shows that $\Theta$ is indeed  locally $d$-to-$1$ near $z_0$, and so $\deg(\Theta, z_0)=\#G_{z_0}$ as desired (for a similar argument in greater generality see the proof of Proposition~\ref{prop:prop_deck_trafo}~\ref{item:pi1_O_2}). 

To  see that 
$\Theta\: \C \ra \CDach$ is a branched covering map, let $z_0\in \C$ be  arbitrary. We choose $B=B_\C(z_0, \eps)$
as before so that \eqref{eq:ghBGz0} holds.  Since $\Theta$ is surjective, it is enough to show that $w_0\coloneqq \Theta(z_0)$ has a neighborhood that is evenly 
covered by 
$\Theta$ as in Definition~\ref{def:brcovmap}. 

 The invariance property of $\Theta$ with respect to the  cyclic rotation group $G_{z_0}$ on $B$  implies that 
 $$ \Theta(z)=f((z-z_0)^d/\eps^d)\quad  \text{for $z\in B$}, $$ 
 where $f\: \D \ra \CDach$ is an injective holomorphic map with $f(0)=w_0$.  
 
 In particular, $V\coloneqq \Theta(B)=f(\D)$ is a  topological
 disk containing $w_0$.  If we define  $\varphi\: B\ra \D$ by $\varphi(z)=(z-z_0)/\eps $ for $z\in B$ and 
 set $\psi \coloneqq f^{-1}\:  V\ra \D$, then $\varphi$ and $\psi$ are orientation-preserving homeomorphisms with $\varphi(z_0)=0$, $\psi(w_0)=0$, and 
\begin{equation} \label{eq:repthzd} 
 (\psi\circ \Theta\circ \varphi^{-1})(u)=u^d
 \end{equation}  for all $u\in \D$. 
 
   Since $\Theta$ is induced by 
$G$, we have by \eqref{eq:ghBGz0} that
$$ \Theta^{-1}(V)=\bigcup_{g\in G}g(B)=\bigcup_{g\in I}g(B),$$
where $I\sub G$ is a set that contains precisely one element from
each left coset of $G_{z_0}$ in $G$. Note that by
\eqref{eq:ghBGz0} this means that the sets in the latter union are pairwise disjoint.
For  each set $g(B)$ with $g\in I$, we have a relation as in \eqref{eq:repthzd} if we replace 
$\varphi\: B\ra \D$ with  $\varphi\circ g^{-1}\: g(B)\ra \D$. This shows that $V$ is evenly covered by $\Theta$. It follows that $\Theta$ is indeed a  branched covering map.

Since $\Theta$ is surjective and induced by $G$, there is a bijection between the set of 
$G$-orbits (i.e., $\C/G$) 
and points in $\CDach$ given by 
$Gz\mapsto \Theta(z)$. As we know, for points $z\in \C$ in a given $G$-orbit, the cardinality
$\#G_z$ is independent of $z$. By using the bijection between $\C/G$ and $\CDach$ we can push 
the well-defined function $Gz\mapsto \#G_z$ over  to a function 
$\alpha \: \CDach\ra \N$ satisfying  $\alpha(\Theta(z))=\#G_z$ for $z\in \C$. If we combine this with what we have seen above, then \eqref{eq:degT_orderG} follows.  
It is also  clear that $\alpha $ is a finite ramification function and that $(\Cdach, \alpha)$ is a parabolic orbifold with a signature corresponding to the type of $G$. This ramification function 
$\alpha$ is uniquely determined as follows from \eqref{eq:degT_orderG} and the surjectivity of $\Theta$.

It remains to show the uniqueness statement for $\Theta$. To this
  end, suppose $\widetilde{\Theta}\colon \C\to S^2$ is another  
  continuous map  induced by $G$.
  Since $G$ acts cocompactly on $\C$, we can apply 
   Corollary~\ref{cor:groupquot}~\ref{item:groupquot2} and conclude that there exist  unique homeomorphisms $\varphi_1\: \C/G \ra \Theta(\C)=\CDach$ and 
  $\varphi_2\: \C/G \ra \widetilde{\Theta}(\C)\sub S^2 $  
 such that $\Theta= \varphi_1 \circ \Theta_G$ and $\widetilde \Theta= \varphi_2 \circ  \Theta_G$. Here $ \Theta_G\: \C \ra \C/G$ is the quotient map. 
 It follows that $\C/G$ and $\widetilde{\Theta}(\C)$ are homeomorphic to $\CDach$. On the other hand,   $\widetilde{\Theta}(\C)\sub S^2$ which is only possible if $\widetilde{\Theta}(\C)= S^2$. So $\varphi_2$ is actually a homeomorphism from  $\C/G$ onto  $S^2$. Then $\varphi\coloneqq \varphi_2\circ \varphi_1^{-1}$ is a homeomorphism from $\CDach$ onto $S^2$ with $\widetilde \Theta = \varphi \circ \Theta$. The uniqueness of $\varphi$ immediately follows from the uniqueness of  $\varphi_2$. 
 
Since $\Theta\: \C \ra \CDach$ is a branched covering map, 
$\widetilde \Theta = \varphi \circ \Theta$ is also a branched covering map. Hence if $S^2=\CDach$ and $\widetilde \Theta$ is holomorphic, then $ \varphi\: \CDach \ra S^2=\CDach$ is  holomorphic by Lemma~\ref{lem:2_3_branched}. It follows that in this case the
homeomorphism $\varphi$ is a M\"obius transformation.   
\end{proof}

\begin{proof}[Proof of Theorem~\ref{thm:orbunifparacas}]
The existence of the  universal orbifold covering map $\Theta$ easily follows from our earlier considerations in this section.   
Namely, suppose  $\mathcal{O}=(\CDach, \alpha)$ is a parabolic orbifold with  
finite ramification function $\alpha\: \Cdach \ra \N$.
Then the  signature of  $\mathcal{O}$ is  $(2,4,4)$, $(3,3,3)$,  $(2,3,6)$, or $(2,2,2,2)$. If $\mathcal{O}$ has only three conical singularities,   we pick a crystallographic group $G$ of a corresponding type  and consider the map $\Theta$  induced by $G$ as provided by Proposition~\ref{prop:G_yields_T}. As follows from \eqref{eq:degT_orderG}, the map 
 $\Theta$ has three critical values. In the fibers over these values the local degree of $\Theta$ is constant and corresponds to the different values $a$, $b$, or $c$ giving  the type $(abc)$ of $G$. 
One can  postcompose this map with a M\"obius transformation so that these three critical values match the three cone points  of the orbifold. In this way, one  obtains a new holomorphic branched covering map $\Theta\: \C \ra \Cdach$ that satisfies 
 \eqref{item:paraorb1}. For type $(244)$ this was discussed in detail at the beginning of this section. 
 
 If $(\CDach, \alpha)$ has signature $(2,2,2,2)$,  then $\Theta$ is a 
 Weierstra\ss\ $\wp$-function  followed by a suitable M\"obius transformation.
We have seen  earlier in this section how to choose the period lattice  $\Gamma$ 
of the $\wp$-function  and the M\"obius transformation so that the four critical values of $\Theta$ match the four cone points  of $\mathcal{O}$. 

As we already remarked, 
the  statements about uniqueness of  $\Theta$ and  its  deck transformation group  follow from more general facts developed in the appendix (see Corollary~\ref{cor:unique_univ_orbi}, Remark~\ref{rem:holoorbset}, and 
Proposition~\ref{prop:prop_deck_trafo}). 
\end{proof}

\section{Examples of Latt\`{e}s  maps}
\label{sec:examples-lattes-maps}

In this section we present some examples of Latt\`{e}s maps
based on Theorem~\ref{thm:Lattesstruc}~\ref{item:Lattessruciii}.

\begin{ex}[A Latt\`{e}s map with orbifold signature $(2,4,4)$]
  \label{ex:Lattes244}
 To  exhibit  such a Latt\`{e}s map, we use the considerations in the beginning of Section~\ref{sec:paraorbcov}. 
  Let $\Delta$  again be the pillow obtained from gluing together two isometric  triangles
  $T_{\wt}$ and $T_{\bt}$ with angles $\pi/2$, $\pi/4$, $\pi/4$. We call $T_\wt\sub \Delta$ the white $0$-tile,  and 
  $T_\bt\subset \Delta$ the black $0$-tile. 

  We  divide $T_\wt$ along the perpendicular bisector of its
  hypotenuse into two triangles $T_1$ and $T_2$ that are similar to
  $T_\wt$ by the scaling factor $\sqrt{2}$. In the same way  we divide  $T_\bt$ into two similar triangles. These four small
  triangles are called $1$-tiles; we color them black and white
  in a checkerboard fashion so that two $1$-tiles with a common 
  edge have distinct colors as indicated on the left in Figure~\ref{fig:lattes244a}. Note that the triangle with the lighter gray shading is on the back of the triangular pillow. 

  \ifthenelse{\boolean{nofigures}}{}{  
  \begin{figure}
    \centering
    \begin{overpic}
      [width=11cm, 
      tics=20]{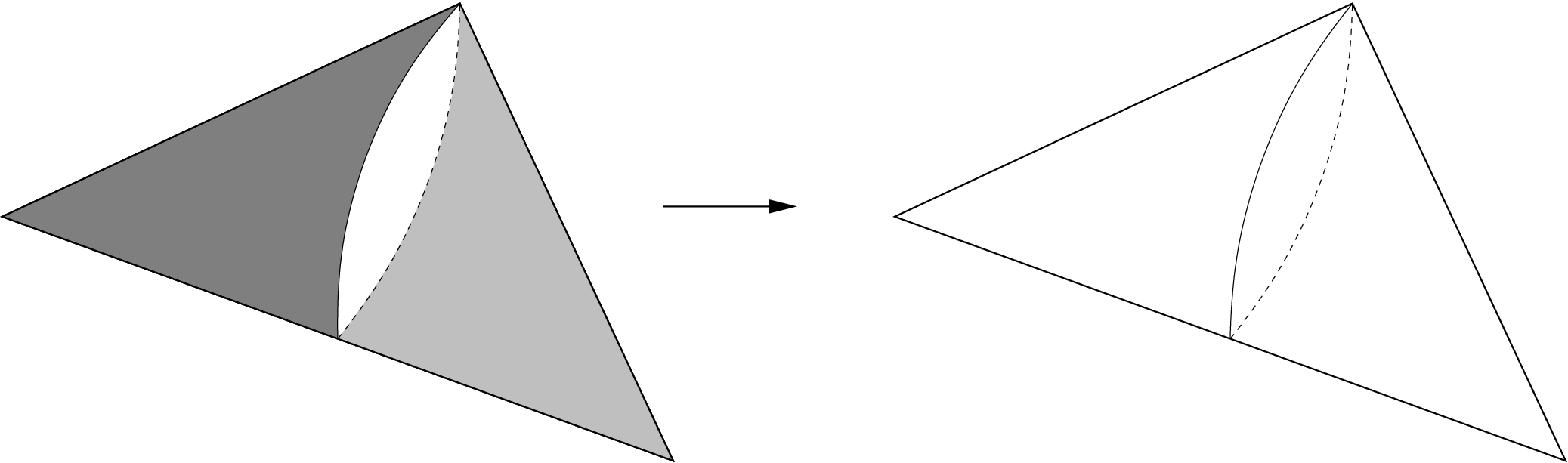}
      \put(53,14){$\scriptstyle{-1}$}
      \put(101,-1){$\scriptstyle{1}$}
      \put(85,30){$\scriptstyle{\infty}$}
      \put(77,5){$\scriptstyle{0}$}
      \put(39,-3){$\scriptstyle{1\mapsto -1}$}
      \put(25,30){$\scriptstyle{\infty\mapsto 1}$}
      \put(18,4){$\scriptstyle{0\mapsto \infty}$}
      \put(-5,11){$\scriptstyle{-1\mapsto -1}$}
      \put(45,18){$\scriptstyle{g}$}
    \end{overpic}
    \caption{A Latt\`{e}s map with orbifold signature $(2,4,4)$.}
    \label{fig:lattes244a}
  \end{figure}
  }

  A map $g\colon \Delta\to \Delta$ is now given as
  follows. We send each white $1$-tile to $T_\wt$ and  each black $1$-tile
  to $T_\bt$ by a similarity as indicated in Figure~\ref{fig:lattes244a}. Then  $g$ is  a branched covering
  map. The points labeled $0$ and $\infty$ are  its critical
  points and we have the following ramification portrait:  
    \begin{equation}
    \label{eq:crit_port_Lattes244}
    \xymatrix @R=1pt{
      0 \ar[r]^{2:1} & \infty \ar[r]^{2:1} & 1\ar[r] &  -1\rlap{.} \ar@(r,u)[] 
    }
  \end{equation}
  Thus $g$ is a Thurston map with $\post(g)=\{-1,1,\infty\}$ and orbifold 
  signature $(2,4,4)$.  Note that $g$ is given as $z\mapsto (1+\iu)z$ in
  suitable Euclidean coordinates (i.e., if we identify $T_\wt$ and $T_\bt$ with a
  triangle $T\subset \C$ in the obvious way, so that the fixed point $-1\in
  \Delta$ is identified with $0\in \C$). In particular, 
   $g\colon
  \Delta\to \Delta$ is holomorphic with respect to the given conformal
  structure on $\Delta$.  
  
 Let $\varphi\colon \Delta\to
  \CDach$ be the  conformal map as considered in Section~\ref{sec:paraorbcov}. Then   $f\colon \CDach\to
  \CDach$ given by $f\coloneqq  \varphi\circ g\circ \varphi^{-1}$ is
  a rational Thurston map. 
   It  is  actually a Latt\`{e}s map by
  Theorem~\ref{thm:Lattesstruc}, because  its orbifold signature  is $(2,4,4)$. If the 
  conformal conjugacy $\varphi$ is chosen such that it sends  the
  points labeled $-1,1,\infty\in \Delta$ to $-1,1,\infty\in\CDach$, respectively, then it is not  hard to check that  $f(z) = 1-2/z^2$.
     
The map  $f$ can also be represented as in
  Theorem~\ref{thm:Lattesstruc}~\ref{item:Lattessruciii}. Namely,  $f$ is
  the quotient of the map $A\colon z\mapsto (1+\iu)z$ by the
  crystallographic group $\widetilde G$ of type (244) in
  Theorem~\ref{thm:G_signature}. 
\end{ex}

\begin{figure}
  \centering
  \begin{overpic}
    [width=12cm, tics=10,
    ]
    {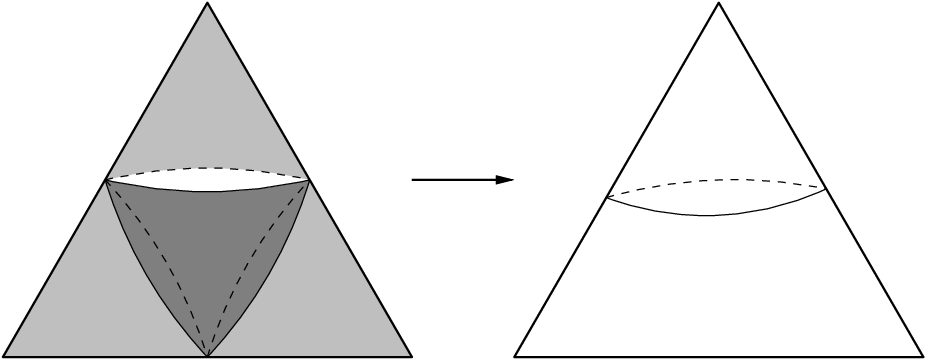}
    \put(-4,-2){$\scriptstyle -1\mapsto \infty$}
    \put(43,-2){$\scriptstyle 1\mapsto 1$}
    \put(21,-2){$\scriptstyle 0\mapsto -1$}
    \put(8,21){$\scriptstyle \mapsto 1$}
    \put(33.5,21){$\scriptstyle \mapsto \infty$}
    \put(21,39.5){$\scriptstyle \infty\mapsto -1$}
    \put(49,21.5){$\scriptstyle f$}
    \put(53.5,-2){$\scriptstyle -1$}
    \put(99,-2){$\scriptstyle 1$}
    \put(76.5,39.5){$\scriptstyle \infty$}
  \end{overpic}
  \caption{A Latt\`{e}s map with orbifold signature $(3,3,3)$. }
  \label{fig:Lattes333}
\end{figure}
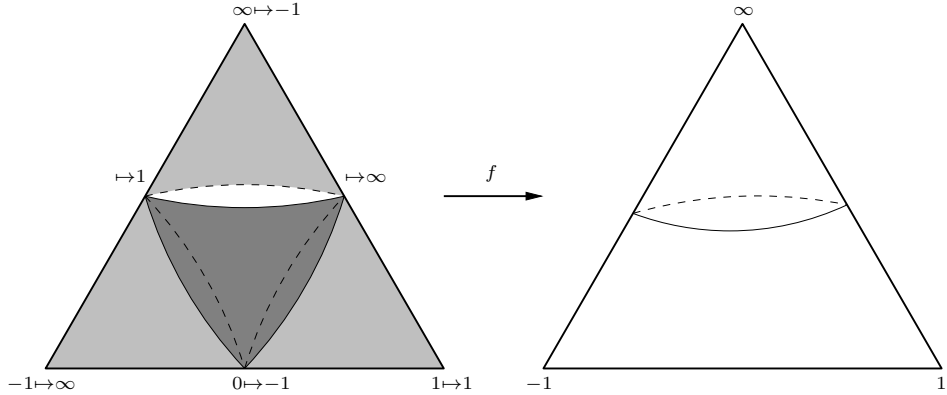

\begin{ex}
  \label{ex:Lattes333}
  In Figure~\ref{fig:Lattes333} we illustrate a Latt\`{e}s map $f$ with
  orbifold signature $(3,3,3)$. The map given here is obtained as a quotient of
  the map $A(z)=2z$ by the crystallographic group $\widetilde G$  of type (333) in
  Theorem~\ref{thm:G_signature}. The Riemann sphere
  $\CDach$ equipped with the 
canonical orbifold metric\index{canonical orbifold!metric}\index{metric!canonical orbifold}\index{o@$\omega$}\index{orbifold!canonical metric} 
is isometric to
  a pillow obtained by gluing together two equilateral triangles
  along their boundaries. The map $f$
  has   in fact  the form 
  \begin{equation*}
    f(z)= 1-2\frac{(z-1)(z+3)^3}{(z+1)(z-3)^3}.
  \end{equation*}
\end{ex}

\begin{figure}
  \centering
  \begin{overpic}
    [width=12cm, tics=10,
    ]
    {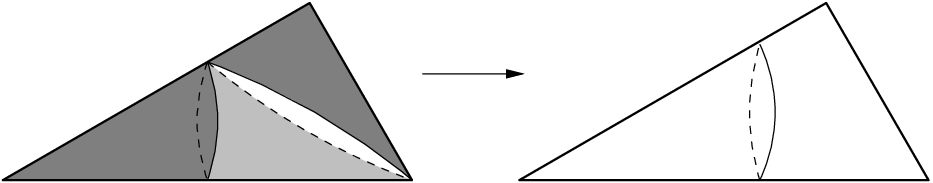}
    \put(-2,-2){$\scriptstyle 0\mapsto 0$}
    \put(22,-2){$\scriptstyle \mapsto \infty$}
    \put(43,-2){$\scriptstyle 1\mapsto 0$}
    \put(31,20){$\scriptstyle \infty\mapsto \infty$}
    \put(20.5,15){$\scriptstyle \mapsto 1$}
    \put(50,13){$\scriptstyle f$}
    \put(55,-2){$\scriptstyle 0$}
    \put(99,-2){$\scriptstyle 1$}
    \put(87,20){$\scriptstyle \infty$}
  \end{overpic}
  \caption{A Latt\`{e}s map with orbifold signature $(2,3,6)$.}
  \label{fig:lattes236}
\end{figure}
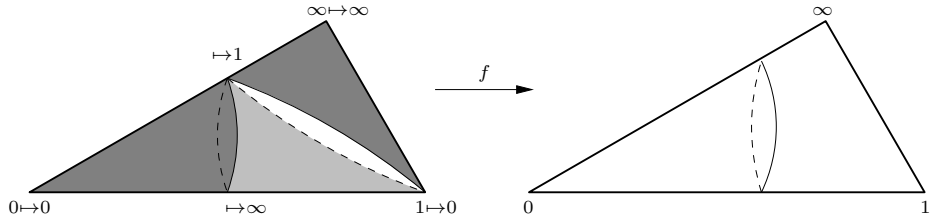

\begin{ex}
  \label{ex:lattes236}
   A Latt\`{e}s map $f$ with orbifold signature
  $(2,3,6)$ is illustrated in  Figure~\ref{fig:lattes236}. It is obtained as the quotient of the map
  $A(z)=\frac 12 (3 + \iu\sqrt 3)z$ by a
  crystallographic group $\widetilde G$ 
  of type (236) in Theorem~\ref{thm:G_signature}. The map $f$ is given by 
  \begin{equation*}
    f(z)=1-\frac{(3z +1)^3}{(9z-1)^2}. 
  \end{equation*}
\end{ex}

Note that the rational maps $f$ discussed in the previous  three examples all  have
the property that the extended real line
$\CC\coloneqq \widehat{\R}=\R\cup\{\infty \}\subset \CDach$ contains
$\post(f)$ and is $f$-invariant in the sense that
$f(\CC) \subset \CC$. This leads to a simple description of the maps in geometric terms. 

In Chapter~\ref{cha:subdivisions} we will consider more general
Thurston maps with such invariant Jordan curves, and in
Chapter~\ref{cha:constructc} we will see that such an invariant
Jordan curve exists for each sufficiently high iterate of an
expanding Thurston map. This means that we obtain similar
geometric descriptions in much greater generality.

 Latt\`es maps with orbifold signature $(2,2,2,2)$ are obtained from $G$-equi\-va\-riant maps $A$ as in Theorem~\ref{thm:Lattesstruc}~\ref{item:Lattessruciii}. 
  The conditions on
$A$ specified in Proposition~\ref{prop:alph_beta_G} are quite restrictive. For 
generic $\tau\in \C$ with $\imag (\tau)>0$ there are no $\alpha\in
\C\setminus\Z$ with $\alpha\Gamma\subset \Gamma$. More
precisely, such  $\alpha\in \C\setminus\Z$ exists if and only if the
lattice allows so-called complex multiplication and $\alpha$ is an
integer in an imaginary quadratic field. 
On the other hand, if $\alpha\in \Z\setminus \{0\}$, then $\alpha\Gamma\sub \Gamma$ for each lattice $\Gamma$.  This leads to a class of Latt\`es maps given  by the following definition. 

\begin{definition}[Flexible Latt\`{e}s map]
  \label{def:flex_Lattes}
  \index{Latt\`{e}s map!flexible} \index{flexible Latt\`{e}s map} A
  Latt\`{e}s map $f\colon \CDach \to \CDach$ is called
  \emph{flexible} if its orbifold  has signature $(2,2,2,2)$ and can be
  represented as in
  Theorem~\ref{thm:Lattesstruc}~\ref{item:Lattessruciii} with a crystallographic group $G=\widetilde G$ of type $(2222)$ as in Theorem~\ref{thm:G_signature} and a map 
  $A\colon \C\to \C$ of the form $A(z)=\alpha z+ \beta$ with 
  $\alpha\in \Z\setminus \{-1,0,1\}$ and $2\beta\in \Gamma$, where $\Gamma$ is the underlying lattice of  $G$.
\end{definition}
Note that we have to rule out $\alpha=\pm 1$ due to the requirement that $\deg(f)
=|\alpha|^2\ge 2$ for a Thurston map $f$. 

The term ``flexible'' derives from the fact that by deforming
the underlying lattice $\Gamma$ (and possibly the parameter
$\beta$ of $A$) of a flexible Latt\`es map, one obtains a family
of such maps depending on one complex parameter (for example,
the parameter $\tau$ in the representation of the lattice
$\Gamma=\Z\oplus \Z \tau$). It is not hard to see that this
leads to maps that are topologically conjugate, but in general
not conjugate by a M\"obius transformation on $\CDach$ (see the
discussion below). An  explicit example of such a family  (with $A(z)=2z$ and $\beta=0$) is given by 
\begin{equation}
  \label{eq:flex_Lattes}
  f(z) = \frac{4z(1-z)(1-k^2z)}{(1-k^2z^2)^2},
\end{equation}
where $k\in \C\setminus\{0,1,-1\}$. The map $g$ discussed in Section~\ref{sec:Lattes} corresponds to $k=\iu$.

Flexible Latt\`{e}s maps are exceptional in
many respects. For example, Thurston's uniqueness theorem (see
Theorem~\ref{thm:Thurstonuniqness}) fails for these maps. They
also carry invariant line fields (see \cite[Chapter 3.5]{McM}).
According to a well-known conjecture they are the only
rational maps with this property. Its validity would imply the
``density of hyperbolic rational maps'', which is possibly the
most famous open problem in complex dynamics (see \cite{MSS} and
\cite{McM_2} for an overview).

Flexible
  Latt\`{e}s maps can be considered as generic Latt\`{e}s
maps, because they are the only type of Latt\`es maps   that can be defined 
for  arbitrary lattices
$\Gamma=\Z\oplus  \Z \tau$  or, equivalently, can be obtained as quotients  of maps on 
arbitrary complex tori $\T$. This is the reason why
some authors use the term ``Latt\`{e}s map'' only for maps with orbifold 
signature $(2,2,2,2)$. For example, Milnor uses this more restrictive definition
in the second edition (and earlier editions), while he uses the
definition used here in the third edition of \cite{Mi}.

To discuss a geometric description of flexible Latt\`es maps,   suppose $f$ is 
such a map
as in Definition~\ref{def:flex_Lattes} obtained from a  map   $A\colon \C\to\C$  of the form
$A(z)=n z+ \beta$ with  $n\in \Z\setminus\{-1,0,1\}$ and $2\beta\in
\Gamma=\Z\oplus \Z\tau$. We get the same map $f$ if we
replace $A$ with  $g\circ A$, where $g\in G$.  This allows
us to assume that $A$ has the form $A(z)=nz+\beta$, where $n\in \N$,
$n\ge 2$, and $\beta\in \{0, 1/2,\tau/2, (\tau+1)/2\}$ (see Proposition~\ref{prop:alpha_beta2}). Note that the map $\Theta_\Delta$ considered in Section~\ref{sec:paraorbcov}
 sends 
the four points $0, 1/2,\tau/2, (\tau+1)/2$ to the four
vertices of the tetrahedron $\Delta$ obtained from a fundamental domain of $G$.

The triangle $T$ from Figure~\ref{fig:tetrahedron} can be divided into $n^2$ triangles that are similar
to $T$ by the scaling factor $n$. For this  we divide  each side of
$T$ into $n$ edges of the same length and draw the line segments
parallel to the sides of $T$ through the endpoints of these edges.
The tetrahedron $\Delta$ has four  faces, denoted  by   $T_1, T_2,T_3
,T_4$. Each of them is similar to $T$, and so we can divide each face of
$\Delta$ in the same way into $n^2$ triangles similar to $T$. In the
following, we will call them the {\em small} triangles in
$\Delta$. 
 
\begin{figure}
  \centering
  \begin{overpic}
    [width=11cm, tics=20,
    ]
    {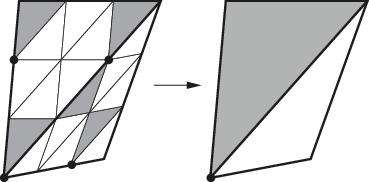}
    \put(46,28){$\widetilde{f}$}
    \put(-3,18){$\Delta$}
    \put(53,18){$\Delta$}
  \end{overpic}
  \caption{A Euclidean model for a flexible Latt\`{e}s map.}
  \label{fig:flex_Lattes}
\end{figure}

We now construct a map $\widetilde{f}\colon \Delta\to\Delta$ as
follows. We consider a small triangle $S$ that contains one of the four
vertices of $\Delta$ and  send  $S$  to one of the faces $T'$ of $\Delta$  by an orientation-preserving similarity
that scales by the factor $n$.  For given $S$ and $T'$ there is only one such map unless the base  triangle $T$, and hence also $S$ and $T'$, are 
 equilateral. In this case,  there are three such maps, only one of which will lead to a flexible
Latt\`{e}s map. We will ignore this special case in order to 
keep our  discussion simple. 

The similarity between $S$ and $T'$  can be uniquely extended to
a continuous  map $\widetilde{f}$ on $\Delta$ that sends  each small
triangle to one of the triangles $T_j$, $j\in \{1,2,3,4\}$,   by an orientation-preserving similarity with  scaling factor  $n$. It is clear that  $\widetilde{f}$
is a branched covering map. Its critical points are the vertices of
the small triangles with the exception of  the four vertices of $\Delta$. The map 
$\widetilde{f}$ sends each vertex
of a small triangle  to one of
the vertices of $\Delta$. It follows that the postcritical points of
$\widetilde{f}$ are the four vertices of $\Delta$ and $\widetilde{f}$
is postcritically-finite. Since the local degree at each critical point is
$2$, it is easy to see that  the orbifold signature of $\widetilde{f}$ is $(2,2,2,2)$. 

If we conjugate $\widetilde{f}$ by a uniformizing  map
$\varphi\: \Delta\ra \CDach$, then we obtain a flexible Latt\`{e}s 
map
 $f\coloneqq \varphi\circ \widetilde{f}\circ
\varphi^{-1}\colon\CDach \to 
\CDach$.    Conversely,  each flexible Latt\`{e}s map can be  obtained in  this
form. One such map (with $n=3$) is illustrated in
Figure~\ref{fig:flex_Lattes} (note that here the tetrahedron
$\Delta$ is embedded in $\R^3$). 

If we use  exactly the same construction as  just described with  another  parameter  $\tau'\in \C$, $\imag(\tau')>0$, then  we obtain a different tetrahedron $\Delta'$.
This leads to a  
different flexible Latt\`es map $g\colon \CDach \to \CDach$. It is easy to see that 
the  $f$ and $g$ are always 
topologically conjugate, but in general not  conjugate by a M\"{o}bius
transformation. So we obtain a $1$-parameter family of Latt\`es  maps with similar dynamics. 

\begin{figure}
  \centering
  \begin{overpic}
    [width=10cm, tics=20, 
    ]
    {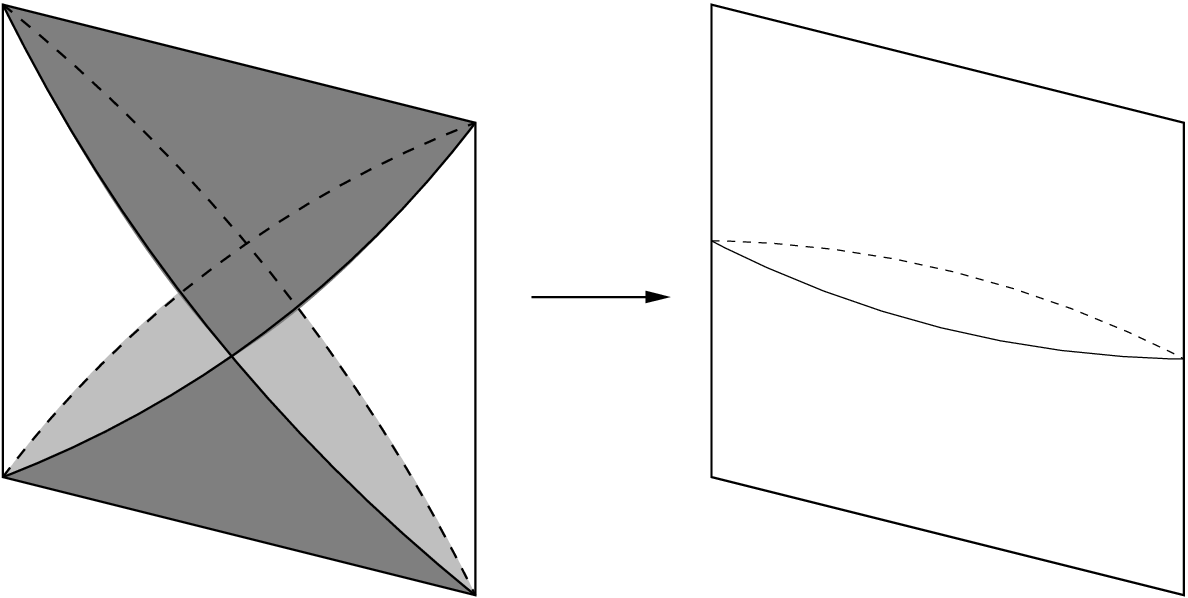}
    \put(-2,7){$\scriptstyle 0\mapsto \infty$}
    \put(-2,51){$\scriptstyle \iu\mapsto 0$}
    \put(38,41.2){$\scriptstyle \infty\mapsto \infty$}
    \put(38,-3){$\scriptstyle -\iu\mapsto 0$}
    \put(21,19){$\scriptstyle 1\mapsto \iu$}
    \put(17,35){$\scriptstyle -1\mapsto -\iu$}
    \put(50,27){$\scriptstyle f$}
    \put(59,7){$\scriptstyle 0$}
    \put(58,51){$\scriptstyle \iu$}
    \put(98,41.2){$\scriptstyle \infty$}
    \put(99,-3){$\scriptstyle -\iu$}
  \end{overpic}
  \caption{The map $f$ in  Example~\ref{ex:Lattes_Milnor}.}
  \label{fig:Lattes_Milnor}
\end{figure}
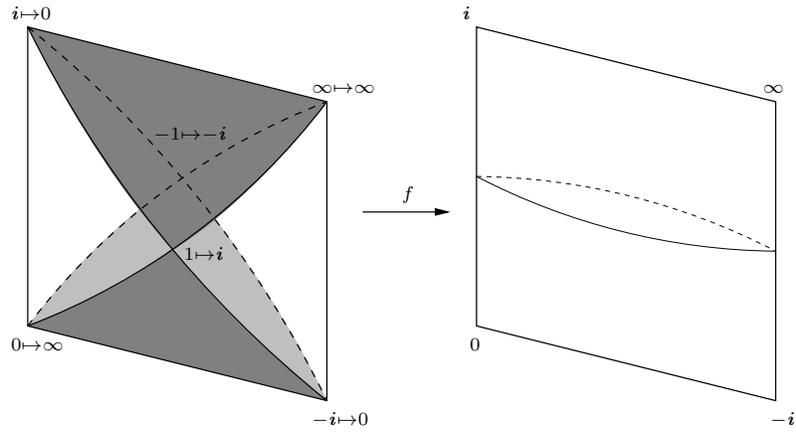

\begin{ex}
  \label{ex:Lattes_Milnor} Not 
   all Latt\`{e}s maps whose  orbifolds have signature
  $(2,2,2,2)$ are flexible Latt\`{e}s maps. For example,  suppose 
  $G$ is the crystallographic group consisting of the maps 
  $g(z)=\pm z + \gamma$,  where $\gamma\in \Gamma\coloneqq \Z\oplus \Z\iu$. 
 Then  $A(z)=(1-\iu)z$ is $G$-equivariant and so   there is a Latte\`s map $f$
  that arises as the quotient of $A$ by $G$ as in 
  Theorem~\ref{thm:Lattesstruc}~\ref{item:Lattessruciii}.  
 Then the  map 
   $f$ has  orbifold signature $(2,2,2,2)$. In fact, 
  $f$ is (up to conjugation by a M\"{o}bius transformation) given by
  $f(z)= \frac{\iu}{2}(z+1/z)$.  Since $\deg(f)=2$, this map is not a flexible Latt\`es map, because the degree of each such map is the square of an integer.
  A geometric model for $f$  is shown in Figure~\ref{fig:Lattes_Milnor}.    
A detailed discussion   of this map can be found in  \cite{Mi04}. 
 \end{ex}

\ifthenelse{\boolean{singlechapter}}{

%
%


\chapter{Quasiconformal and rough geometry}\label{cha:QCRoughGeo}

 The immediate purpose of this   chapter is to provide a framework for  discussing  Cannon's conjecture in geometric group theory and related questions. This will give  some motivating background for our study of expanding Thurston maps. 
 The chapter   will  also  serve as a reference for  some geometric  terminology and facts that we will use throughout   this work.

\section{Quasiconformal geometry} \label{sec:QCgeom}

  In this section we record  some material   related to quasiconformal geo\-metry and the analysis on metric spaces
(see \cite{He} for an exposition  of this subject).

Let $(X,d_X)$ and $(Y,d_Y)$ be metric spaces, and $f\:X\ra Y$ be
a homeomorphism.  Then $f$ is called 
{\em  bi-Lipschitz}\index{bi-Lipschitz} if there exists a constant $L\ge 1$ such that 
$$ \frac 1L d_X(u,v)\le  d_Y(f(u), f(v))\le L d_X(u,v)$$
for all $u,v\in X$. If there exist $\alpha>0$ and $L\ge 1$ such that  
\begin{equation}
  \label{eq:def_snowflake_eq}
   \frac 1L d_X(u,v)^\alpha \le  d_Y(f(u), f(v))\le L d_X(u,v)^\alpha  
\end{equation}
for all $u,v\in X$, then $f$ is called a 
{\em snowflake homeomorphism}\index{snowflake homeomorphism}. 
The map $f$ is  called a {\em quasisymmetric homeomorphism} or a {\em quasisymmetry}\index{quasisymmetry} if there exists 
a homeomorphism  $\eta\:[0,\infty)\ra [0,\infty)$  such that 
\begin{equation}\label{qsdef}
\frac{d_Y(f(u),f(v))}{d_Y(f(u),f(w))}
\leq \eta\left (\frac{d_X(u,v)}{d_X(u,w)}\right)
\end{equation}
for all $u,v,w\in X$ with  $u\ne w$. If we want to emphasize $\eta$
here, then we speak of an {\em $\eta$-quasisymmetry}.

The inverse of a bi-Lipschitz homeomorphism  and the composition
of two such  maps (when it is defined) 
are again a bi-Lipschitz homeo\-morphisms. 
Similarly, the classes of snowflake homeomorphisms and quasisymmetries  are closed under taking inverses or compositions. 
This implies that all these classes of maps lead to a corresponding notion of equivalence for metric spaces. 
So $X$ and $Y$ are called 
{\em bi-Lipschitz}
, {\em snowflake}\index{snowflake equivalent|textbf}, or {\em quasisymmetrically equivalent}
if there exists a homeomorphism between $X$ and $Y$ that is of the corresponding type. In the same way, we call two 
metrics  
 $d$ and $d'$  on a space $X$ {\em bi-Lipschitz}, {\em snowflake}, or {\em quasisymmetrically equivalent}  if the identity map from $(X,d)$ to $(X,d')$ is a homeomorphism with  the corresponding property.  Since every bi-Lipschitz homeomorphism is also a snowflake homeomorphism and every snowflake homeomorphism
is a quasisymmetry, this leads to correspondingly weaker notions of equivalence. 

\index{gauge}
A collection of metrics on a given space $X$ that are mutually  quasisymmetrically equivalent is called a  {\em quasisymmetric gauge} on $X$. One defines a  {\em bi-Lipschitz} or  {\em snowflake gauge}  on $X$ similarly (see \cite[Chapter~15]{He} for  related terminology and further discussion).  

As we will later see in Chapter~\ref{cha:visual-metrics}, each expanding Thurston map $f\:S^2\ra S^2$ induces a natural snowflake gauge on its underlying $2$-sphere $S^2$.  This is related to the   fact  that if $X$ is a  Gromov hyperbolic space, then there is a natural snowflake gauge on its boundary at infinity $\geo X$
(see Section~\ref{sec:Grhyp} below). In both cases we will later use the term {\em visual metric} for a metric in the natural snowflake gauge. A visual metric $\varrho$ gives the underlying  $2$-sphere $S^2$ of an 
expanding Thurston map $f\:S^2\ra S^2$ some geometric structure related to the dynamics of the map. 
 An important question in this context is whether this  sphere $(S^2,\varrho)$ is quasisymmetrically equivalent to the standard $2$-sphere, i.e., the Riemann sphere equipped with the chordal metric (see 
 Theorem~\ref{thm:S2vsf}~\ref{item:S2qsphere}). 
This is an instance  of a more general problem that can be formulated 
as follows.

\begin{quasi_u_problem*}
  \index{quasisymmetric uniformization problem}
  Suppose $X$ is a metric space homeomorphic 
  to some ``standard'' metric space $Y$. 
  When is $X$ quasisymmetrically equivalent to $Y$?
\end{quasi_u_problem*}

This problem is very similar to questions in classical uniformization theory where one asks whether a given region in $\CDach$ or a given 
Riemann surface is conformally 
equivalent to a ``standard'' region or Riemann surface. 
The quasisymmetric uniformization problem can be seen as a metric space version 
of this  where the class of conformal maps is replaced with  
the  class of quasisymmetric
homeomorphisms.

 It  depends 
on the context  how the term ``standard'' is precisely 
interpreted. A satisfactory answer to the quasisymmetric uniformization problem for given $Y$ is essentially equivalent to cha\-racterizing $Y$ up to quasisymmetric equivalence.   We will see 
that this is related to Cannon's conjecture  in 
geometric group theory. 
One can also pose a similar problem  for other classes of maps such as   bi-Lipschitz or snowflake maps. 

The prime instance for a quasisymmetric uniformization result is a 
theorem due to Tukia and 
V\"ais\"al\"a   that characterizes {\em 
quasicircles} and {\em quasiarcs}.  
 In order to state this result, we first have to discuss some terminology.

 The {\em standard unit circle} $\Sph^1$ is the  unit circle in
 $\R^2$ equipped with (the restriction of) the  Euclidean metric. A {\em metric circle}, i.e., a metric space homeomorphic to $\Sph^1$,  is 
called a 
{\em quasicircle}\index{quasicircle}
if it is quasisymmetrically equivalent to $\Sph^1$.
Similarly, a {\em metric arc}   (a metric space homeomorphic to the
unit interval $[0,1]\sub \R$) is called a 
{\em quasiarc}\index{quasiarc}
if it is quasisymmetrically equivalent to $[0,1]$ (equipped with
the Euclidean metric).  

A \emph{continuum}\index{continuum}  is a compact and 
connected topological space.  If a continuum   contains two
distinct points, then  it is called 
{\em non-degenerate}.\index{continuum!non-degenerate}\index{non-degenerate continuum}  
If  $x$ and $y$ are points in a continuum $X$, then we say that $X$ 
 \emph{connects} or  \emph{joins} $x$ and $y$. 

Let $(X,d)$ be a metric space.  Then $X$ is called  
{\em doubling}\index{doubling}\index{metric!doubling}
if there exists a number $N\in \N$ such that every 
ball in $X$ of radius $R>0$ can be covered by 
$N$ balls of radius $R/2$. 
We say that $X$ is of 
{\em bounded turning}\index{bounded turning}\index{metric!bounded turning}
if there  is a constant
  $K\ge 1$ such that for all points $x,y\in X$
  there
  exists a continuum $\ga \sub X$ with $x,y\in \ga$ 
  and
  \begin{equation} \label{eq:bddturn} 
   \diam_d(\ga)\le Kd(x,y).
  \end{equation}
  
  If $X$ is a metric circle, then it is of bounded turning precisely when for all $x,y\in X$ the last inequality is  true for one of the (possibly degenerate) subarcs $\ga\sub X$ with endpoints $x$ and $y$. Similarly, if $X$ is  a metric arc, then it is of bounded turning precisely when for all $x,y\in X$ inequality  \eqref{eq:bddturn}  is  true for the  subarc $\ga$  of $X$
  with endpoints $x$ and $y$.

The Tukia-V\"ais\"al\"a theorem \cite[Theorem~4.9, p.~113]{TV} can now be formulated as follows. 

\begin{theorem} \label{thm:TV} Let $X$ be a metric circle or a metric arc. Then $X$ is a quasicircle or a quasiarc, respectively, 
 if and 
only if $X$ is  doubling and of bounded turning. 
\end{theorem}

A metric space $X$ is called 
{\em linearly locally connected}\index{linearly locally connected (LLC)} 
if there
exists a constant $\lambda\ge 1$ satisfying the following
conditions: if $B(a,r)$ is an open ball in $X$ and
$x,y\in B(a,r)$, 
then there exists a continuum in
$B(a,\lambda r)$ connecting $x$ and $y$.  Moreover, if
$x,y\in X\setminus B(a,r)$, 
then there exists a
continuum in $ X\setminus B(a,r/\lambda)$ connecting $x$ and $y$.

It is easy to see  that a metric circle  is of bounded turning  if and only if it is linearly locally connected. This gives an alternative version of part of the Tukia-V\"ais\"al\"a theorem: {\em a metric circle is a quasicircle if and only if it is doubling and linearly locally connected}. 
 
 Every subset of $\CDach$  (equipped with the restriction of the 
 chordal metric $\sigma$) is doubling; this implies that a Jordan curve $J\sub \CDach$ is a 
 quasicircle if and only if it is of bounded turning. A similar remark applies to arcs $\alpha\sub \CDach$. 

A quasisymmetric characterization of the standard 
$1/3$-Cantor set can be found in \cite[Proposition~15.11]{DS}. 
Work by Semmes \cite{Se0, Se1} shows that the quasisymmetric characterization 
of $\R^n$ or  the standard sphere $\Sph^n$
 for $n\ge  3$  is  a problem that seems to be 
 beyond reach at the moment.
 The intermediate 
case $n=2$ is particularly interesting.  
To formulate a specific result in this direction we need one more definition.

 Let $(X,d)$  be a 
 locally compact metric space, and $\mu$  a Borel measure on $X$. Then  the metric measure space $(X,d,\mu)$   is called   \defn{Ahlfors}
$Q$\defn{-regular},\index{Ahlfors regular|textbf}   
where $Q>0$,  if  
\begin{equation}\label{eq:Ahlreg}
 \frac 1C R^Q\le  \mu(\overline B(x,R))\le  C R^Q
 \end{equation}  
for all closed balls $\overline B(x,R)$ with $x\in X$ and $0<R\leq \diam (X)$, where $C\ge 1$ is independent of the ball. 
If  $(X,d)$ is understood, we also say that $\mu$ is Ahlfors $Q$-regular.\index{H_Q@$\mathcal{H}^Q_d$}\index{Hausdorff!measure}

 If this condition is satisfied, then the $Q$-dimensional Hausdorff measure $\mathcal{H}^Q$ actually satisfies an  inequality as in \eqref{eq:Ahlreg}, and $\mu$ and  $\mathcal{H}^Q$ are comparable and in particular mutually absolutely continuous with respect to each other (for the definition of Hausdorff $Q$-measure $\mathcal{H}^Q$ see \cite[Section~11.2]{Fo99}, for example). If we want to emphasize the underlying metric $d$, we use the notation   $\mathcal{H}_d^Q$ for the Hausdorff $Q$-measure. 
 Note that  every Ahlfors regular measure $\mu$ on $(X,d)$ is a 
{\em doubling measure},\index{doubling!measure}\index{measure!doubling} 
i.e., there exists a constant $C\ge 1$ such that 
$$ \mu(\overline B(x,2R))\le C  \mu(\overline B(x,R))$$ whenever $x\in X$ and $R>0$.

A
 locally compact metric space $(X,d)$  is called  \defn{Ahlfors}
$Q$\defn{-regular}  
for  $Q>0$  if  the $Q$-dimensional Hausdorff measure $\mathcal{H}^Q$ has this property. Since every Ahlfors regular space admits a doubling measure, it is doubling (as a metric space) (see \cite[B.3.4 Lemma, p.~412]{Se4} for this last implication).

For a metric $2$-sphere  $X$ to be quasisymmetrically 
equivalent to the standard $2$-sphere $(\CDach, \sigma)$  it is 
necessary that $X$ is linearly locally connected. 
This  alone is not sufficient, but will be if we add Ahlfors $2$-regularity as an assumption  \cite[Theorem~1.1]{BK}. 

\begin{theorem}\label{mainthm}
 Suppose $X$ is a metric space homeomorphic to $\CDach$. If 
$X$ is linearly locally connected  and  Ahlfors $2$-regular, then $X$ is quasisymmetrically equivalent
to $(\CDach, \sigma)$.
\end{theorem}

A similar result for other simply connected surfaces 
was obtained 
by K.~Wildrick \cite{Wi}.

The assumption of Ahlfors regularity for  some
exponent $Q\ge 2$ 
 is quite natural, because it is satisfied in many interesting cases: for example, for boundaries of Gromov hyperbolic groups (see 
\cite{Co}) or  for $2$-spheres equipped with visual metrics 
of expanding Thurston maps 
(under some mild extra conditions; see Proposition~\ref{prop:Ahlforsreg}).  
There are 
metric $2$-spheres $X$ though that are linearly locally connected  and $Q$-regular with $Q>2$,  but are not
quasisymmetrically equivalent  to $(\CDach, \sigma)$
\cite{Va2}.  

 We will now prove  a fact that will be useful in
 Chapter~\ref{cha:latt-maps-comb}, where Theorem~\ref{mainthm} is
 applied to give a characterization of Latt\`es maps.

 \begin{prop} \label{prop:abscont2reg}
Let $d$ be a metric on $\CDach$ that is quasisymmetrically equivalent 
to the chordal metric $\sigma$. Let $\mu$ be a Borel measure on 
$\CDach$ such  that $(\CDach, d, \mu)$ is Ahlfors $2$-regular.
Then $\mu$ 
and Lebesgue measure 
$\leb$ on $\CDach$ are 
absolutely continuous with respect to each other. 
\end{prop}

\begin{proof} There is no simple justification of this fact. Accordingly, we have to rely 
on concepts and statements that we will not explain in detail, but rather refer to the literature.  

The quasisymmetric equivalence of $\sigma$ and $d$ implies that a metric
ball with respect to $\sigma$ roughly looks like a metric ball with respect to $d$ (typically with quite different radius); more precisely, 
there exists a constant $\lambda\ge 1$ such that for all 
$p\in \CDach$ and all $R>0$ there exists $R'>0$ such that 
$$ B_d(p, R'/\lambda)\sub B_\sigma(p, R) \sub B_d(p, \lambda R'). $$

One can deduce from this and the Ahlfors $2$-regularity  of $(\CDach, d, \mu)$ that $\mu$ is a  doubling measure on 
$(\CDach, \sigma)$.  Actually,  $\mu$ is a  
{\em metric doubling measure}\index{metric!doubling!measure}\index{doubling!measure!metric}
 on $(\CDach, \sigma)$: if $x,y \in \CDach$ are arbitrary, $R\coloneqq \sigma(x,y)$, and we set 
$$d'(x,y)\coloneqq  \mu(B_\sigma(x, R)\cup B_\sigma(y,R))^{1/2},$$
then $d'$ is comparable to a metric (in fact to $d$, meaning that
$d'\asymp d$). See \cite[Chapter 14]{He} for the terminology employed here.

It is known that a metric doubling measure $\mu$ on $(\CDach, \sigma)$  is  absolutely continuous with respect to Lebesgue measure $\leb$ on $\CDach$ with a Radon-Nikodym derivative $w$  that is positive almost everywhere (in fact $w$ is a so-called $A_\infty$-weight; see \cite{Se-1}). Then $d\mu=w\, d\leb$.   Since $w$ does not  vanish on a set of positive Lebesgue measure, this  implies that 
$\leb$ is also absolutely continuous with respect to $\mu$. 
\end{proof}

A homeomorphism $f\: X\ra Y$ between metric  spaces
$(X,d_X)$ and $(Y,d_Y)$ is called 
{\em weakly quasisymmetric}\index{weak quasisymmetry}\index{quasisymmetry!weak}
 if there exists a constant $H\ge 1$ such that for all $u,v,w \in X$ the following implication holds:
$$ d_X(u,v)\le d_X(u,w) \Rightarrow  
d_Y(f(u),f(v))\le H d_Y(f(u),f(w)).
$$
Under mild extra assumptions on the spaces, a weak quasisymmetry
is in fact a quasisymmetry (\cite[Theorem~10.19]{He}).

\begin{prop} \label{prop:wkqs}
Let $(X,d_X)$ and $(Y,d_Y)$ be metric spaces, and $f\:X\ra Y$ be
a weakly quasisymmetric
homeomorphism. If $X$ and $Y$ are connected and doubling, then $f$ is
a  quasisymmetry.   
\end{prop}

This proposition is very useful if one wants to establish that a given map is a quasisymmetry. 

\index{quasiregular}
Let $U$ and $V$ be open regions  in 
$\CDach$. A continuous map  $f\: U\ra V$  is called 
{\em quasiregular} if $f$ belongs to the Sobolev space 
$W^{1,2}_{loc}$ (i.e., $f$ has weak partial derivatives in $L^2_{loc}$ on $U$)
and if there exists $K\ge 1$ such that 
$$ \Vert Df(p) \Vert_\sigma^2 \le K \det (Df(p))$$ 
for almost every $p\in \CDach$. Here $Df(p)$ denotes the (formal) derivative of $f$ at $p$ considered as a linear map between tangent spaces and   $\Vert Df(p) \Vert_\sigma$ denotes the operator norm 
of $Df(p)$ with respect to the spherical metric (see  \eqref{eq:sphder2}). If, in addition, $f$ is a homeomorphism between 
$U$ and $V$, then $f$ is called 
{\em quasiconformal}.\index{quasiconformal} 
If we want to emphasize the parameter $K$, then we call such a map $K$-{\em quasiconformal} or $K$-{\em quasiregular}. 
For more background on quasiconformal and quasiregular maps see 
\cite{Va} and \cite{Ri}. 

If one considers a family $\mathcal{F}$  of maps in one of the
classes that we discussed, then it is often important to know
whether the distortion properties of the maps are controlled by
the same parameters. If this is the case, we say that
$\mathcal{F}$ is a 
{\em uniform}\index{uniform} 
family of maps in the given class. For example, a family 
$\mathcal{F}$ of homeomorphisms is said to consist of 
{\em uniform quasisymmetries}\index{uniformly!quasisymmetric}\index{quasisymmetry!uniform} 
if there exists a homeomorphism $\eta\: [0,\infty)\ra [0,\infty)$ such that each map in $\mathcal{F}$ is an $\eta$-quasisymmetry. Similarly, we say that the maps in $\mathcal{F}$
are 
{\em uniformly quasiregular}\index{uniformly!quasiregular}\index{quasiregular!uniformly} 
if there exists $K\ge 1$ such that 
each $f\in \mathcal{F}$ is $K$-quasiregular, etc. 

We conclude this section with a discussion of Hausdorff distance and Hausdorff convergence. Let $(X,d)$ be a  metric space. If $A,B\sub X$ are subsets of
$X$, then their {\em Hausdorff distance}
\index{Hausdorff!distance}  
is defined as  
\begin{equation}
  \label{eq:def_Hausdorffd}
  \dist^H_d(A,B)
  \coloneqq 
  \inf\{\delta>0: A\sub \mathcal{N}_{\delta} (B)
  \text{ and } 
  B \sub \mathcal{N}_{\delta} (A)\} \in [0,\infty].   
\end{equation}
Here  for $\delta>0$,
$$\mathcal{N}_\delta(M) = \mathcal{N}_{d,\delta}(M)\coloneqq \{x\in X: \dist_d(x,M)<\delta\}$$ denotes
the $\delta$-neighborhood of a set $M\sub X$. 
If   $A$ and $A_n$ for $n\in \N$ are non-empty closed subsets of $X$, then  we say
that $A_n\to A$  {\em in the sense of Hausdorff
  convergence}\index{Hausdorff!convergence} or $A_n$ {\em Hausdorff converges}  to $A$ as $n\to \infty$   if  
$$ \lim_{n\to \infty}\dist^H_d(A_n,A)=0.$$
Note that in this case a point $x\in X$ lies in $A$ if and only
if there exists a sequence $\{x_n\}$ of points in $X$ such that
$x_n\in A_n$ for $n\in \N$ and $x_n\to x$ as $n\to \infty$. 

It is known that if $X$ is compact, then the 
space of all non-empty closed subsets of  $X$ equipped with the
Hausdorff distance is a complete metric space (see \cite[Section~7.3.1]{BBI}).

%
%
%

\section{Gromov hyperbolicity} \label{sec:Grhyp} 
 
 In this section we 
review some standard material on Gromov hyperbolic spaces. For general background on this topic see   \cite{BuS, gh, Gr}. 

Let $(X,d)$ be a metric space. Then for $p,x,y\in X$ the quantity 
\begin{equation}
  \label{eq:def_Grpr}
  (x\cdot y)_p 
  \coloneqq  
  \frac12 \big(d(x,p)+d(y,p) - d(x,y)\big)
\end{equation}
is called the 
{\em Gromov product}\index{Gromov!product|textbf}\index{$3$@$(x\cdot
  y)_p$}
 of $x$ and $y$ with respect to the basepoint $p$. 
The space   $X$ is called {\em $\delta$-hyperbolic} for  $\delta\ge 0$, 
if the inequality 
\begin{equation}\label{def:Grprod}
(x\cdot z)_p \ge  \min\{ (x\cdot y)_p, (y\cdot z)_p\}-\delta
\end{equation} 
holds for all $x,y,z,p\in X$.  
If this condition is true for some basepoint $p\in X$ (and all
$x,y,z\in X$), then it is actually true for all basepoints if one
changes the constant $\delta$ to $2\delta$.  We say that $X$ is 
{\em Gromov hyperbolic}\index{Gromov!hyperbolic}\index{hyperbolic!Gromov}
if $X$ is $\delta$-hyperbolic for some $\delta\ge
0$.  

The space $X$ is called {\em geodesic}
if  any two points in $X$ can be
joined  by a path whose length is equal to the distance of the points. For a geodesic metric space Gromov hyperbolicity  
is equivalent to a thinness condition for geodesic triangles \cite[Chapter~2]{gh}.

Roughly speaking, the  Gromov hyperbolicity   of a space requires it to be    
``negatively curved'' on large scales. 
Examples for  such spaces are simplicial trees,  
or Cartan-Hadamard manifolds with a negative upper curvature bound such as  (real) hyperbolic $n$-space
${\mathbb H}^n$, $n\ge 2$. 

Let  $(X,d_X)$ and $(Y,d_Y)$ be  metric spaces. A map
$f\:X\rightarrow Y$ is called a 
{\em quasi-isometry}\index{quasi-isometry}
if there exist constants $\lambda\ge 1$ and $k\ge 0$ 
such that 
\begin{equation} \label {eq:qisom1}
\frac 1\lambda  d_X(u,v)  -  k  \le d_Y( f(u),f(v)) 
\le \lambda d_X(u,v)  +  k 
\end{equation} 
for all $u,v\in X$ and if 
\begin{equation}\label{eq:cobdd} 
\inf_{x\in X}d_Y(f(x),y)\le k
\end{equation}
for all $y\in Y$. If $\lambda=1$, then we call $f$ a {\em rough-isometry}.\index{rough-isometry} 
The spaces $X$ and $Y$ are called {\em quasi-isometric} or {\em rough-isometric} if there exists a map  $f\: X \ra Y$ that is a quasi-isometry or a rough-isometry, respectively. In coarse geometry  one often considers two metric spaces the same if they are quasi-isometric or rough-isometric.

Quasi-isometries form  a natural class 
of maps in the theory of Gromov hyperbolic spaces. 
For example, Gromov hyperbolicity of geodesic metric spaces is invariant under quasi-isometries \cite[Chapter~5]{gh}. 

A subset $A$ of a  metric space $(X,d)$ is called 
{\em cobounded}\index{cobounded} 
if there exists a constant $k\ge 0$ such that for every $x\in X$ there exists $a\in A$ with $d(a,x)\le k$. Then   every point in $X$ lies within uniformly bounded distance of the set $A$.
With this terminology,  condition \eqref{eq:cobdd} says that
the map $f$ has cobounded image in $Y$. If $A$ is cobounded in $X$, then $X$ is Gromov hyperbolic if and only if $A$ (equipped with the restriction of the ambient metric) is Gromov hyperbolic. 

With  each Gromov hyperbolic space $X$ one can associate 
a {\em boundary at infinity}\index{boundary at infinity}
$\geo X$\index{d -X@$\geo X$} 
as follows. We fix a basepoint $p\in
 X$, and consider sequences of points $\{x_i\}$ in $X$ 
{\em converging to infinity}\index{convergence to infinity}
in the sense that 
 \begin{equation}
   \label{eq:defxi_infty}
   \lim_{i,j\to\infty} (x_i\cdot x_j)_p=\infty.
 \end{equation}
We declare two such sequences $\{x_i\}$ and 
$\{y_i\}$ in $X$ as {\em equivalent}
if
\begin{equation}
  \label{eq:defxy_eqiv}
  \lim_{i\to\infty} (x_i\cdot y_i)_p=\infty.
\end{equation}
Then    $\geo X$  is defined as the 
set of equivalence classes of sequences 
converging to infinity. It is easy to see that the choice of the basepoint $p$ does not matter here. Moreover, if $A$ is cobounded in $X$, then one can represent each equivalence class by a sequence in $A$, and so we have a natural identification 
$\geo X\cong\geo A$.

A metric space $X$ is called 
{\em proper}\index{proper!metric space} 
 if every closed ball in $X$ is compact. If a  Gromov hyperbolic space $X$ is proper and geodesic, then 
there is an equivalent definition of $\geo X$ as 
the set of equivalence classes of geodesic rays 
emanating from the basepoint $p$.  
One declares two such rays as equivalent if they stay within bounded Hausdorff distance. Intuitively, 
a ray represents its ``endpoint''
in $\geo X$ (see \cite[Section~2.4.2]{BuS} or
\cite[Section~III.H.3] {BH} for details).

The Gromov product on a Gromov hyperbolic space $X$ has a natural extension to the boundary $\partial_\infty X$. Namely, if $p\in X$ and $a,b\in \geo X$, we set
\begin{equation}\label{def:Grprodinfty}
 (a\cdot b)_p\coloneqq  \inf \big\{\liminf_{i\to \infty} (x_i\cdot y_i)_p: \{ x_i\}\in a, \, \{y_i\}\in b\big\}. 
\end{equation}
Note that by definition a point in $\partial_\infty X$  is an equivalence class
(i.e., a set) 
 of sequences in $X$. So in \eqref{def:Grprodinfty}  it makes sense to take the infimum   over all sequences 
 $\{ x_i\}$ and $\{y_i\}$ that represent (i.e., are contained in) $a$ and $b$, respectively. We have $(a\cdot b)_p\in[0,\infty]$, where $(a\cdot b)_p=\infty$ 
 if and only if $a=b$. 
 
  In   \eqref{def:Grprodinfty} one can actually use any sequences representing the points $a$ and $b$ to determine $(a\cdot b)_p$ (up to an irrelevant additive constant). 
 Namely,  there exists a constant $k\ge 0$ independent of $a$ and $b$ such that 
  for all $\{ x_i\}\in a$  and $ \{y_i\} \in b$ we have 
  \begin{equation}\label{def:Grprodinfty2}
 \liminf_{i\to \infty} (x_i\cdot y_i)_p-k \le (a\cdot b)_p\le   \liminf_{i\to \infty} (x_i\cdot y_i)_p.
\end{equation}

The boundary $\geo X$ is  equipped with a natural class of   visual metrics. By definition 
a metric $ \varrho$ on $\geo X$ is called {\em visual}
\index{visual metric!on $\geo X$}  
if there exists a basepoint $p\in X$, and constants $C\ge 1$ and $\Lambda>1$ 
such that 
\begin{equation} \label{vismetric}
\frac{1}{C} \Lambda^{-(a \cdot b)_p}  \le
\varrho(a,b)  \le C  \Lambda^{-(a \cdot b)_p}
\end{equation}
for all $a,b\in \geo X$ (here we use   the convention that
$\Lambda^{-\infty}=0$). We call $\Lambda$ the \emph{visual
  parameter} of $\varrho$.  
 If  $X$ is  $\delta$-hyperbolic, then there exists a
visual metric $\varrho$  for each visual
parameter $\Lambda>1$ sufficiently close to $1$. 

Later we will also define a notion of a visual metric for an expanding
Thurston map.  So then  we have two notions of visual metrics---one for expanding
Thurston maps and one for boundaries of Gromov hyperbolic spaces. 
 For clarity we will sometimes use the phrases {\em visual metric
   in the sense of Thurston maps} and {\em visual metric in the
   sense of Gromov hyperbolic spaces}  to distinguish between
 these two notions. 
We will see that with each expanding Thurston map $f\: S^2\ra S^2$ one can associate a Gromov hyperbolic space 
$\G$ such that $\partial_\infty \G$ can be identified with $S^2$ and such that 
the class of visual metrics for $f$ (in the sense of Thurston maps) 
is exactly the same as the class of visual metrics on
$\partial_\infty \G\cong S^2$ (in the sense of Gromov hyperbolic
spaces); see Theorem~\ref{thm:visualGrTh}. This fact was actually the reason for our choice of the term ``visual metric'' for expanding Thurston maps.

We can think of the boundary at infinity of a Gromov hyperbolic
space $\geo X$ as a metric space if we equip it with a fixed
visual metric.  If $\varrho_1$ and $\varrho_2$ are two visual
metrics on $\geo X$, then the identity map on $\geo X$ is a snowflake
equivalence between $(\geo X, \varrho_1)$ and
$(\geo X, \varrho_2)$. So the visual metrics form a snowflake
gauge on $\geo X$. In particular, $\geo X$ carries a well-defined
topology induced by any visual metric, and the ambiguity of the
visual metric is irrelevant if one wants to speak of snowflake or
quasisymmetric maps on $\geo X$. One should consider the space
$ \geo X$ equipped with such a visual metric $\varrho$ as being
very ``fractal''. For example, assume there exists a visual metric
$\varrho_0$ with visual parameter $\Lambda_0>1$. Then for every
visual metric $\varrho$ with a visual parameter $\Lambda$ that
satisfies $1< \Lambda< \Lambda_0$, the space $(\geo X, \varrho)$
does not contain any non-constant rectifiable curves.

The following fact links the theory of Gromov hyperbolic spaces   to 
quasisymmetric maps (see \cite{BS} for more on this subject). 
\begin{prop} \label{prop:qisomqs}
Let $X$ and $Y$ be proper and geodesic Gromov hyperbolic spaces. 
Then every quasi-isometry $f\: X\ra Y$  induces  
a natural quasisymmetric  boundary map $\tilde f\:\geo X \ra \geo Y$.
\end{prop} 
The boundary map $\tilde f$ is defined 
by assigning to a point $a\in \geo X$ represented by the sequence $\{x_i\}$ the point $b\in \geo Y$ represented by the sequence $\{f(x_i)\}$.  

This statement  lies at the heart of Mostow's proof 
for  rigidity of rank-one symmetric spaces \cite{Mo}. 
The point is that a quasi-isometry may locally 
exhibit very irregular behavior, but  gives rise to a quasisymmetric boundary map that can be analyzed by analytic tools. 

\section{Gromov hyperbolic groups and Cannon's conjecture}
\label{sec:Cannconj}

 The theory of Gromov hyperbolic spaces can be used to define 
a  class of 
discrete groups. Here one adopts a geometric point of view by studying the Cayley graph of the group.  We 
review some standard definitions related to this,   
but will not attempt an in-depth treatment of the subject (for more details, see \cite{gh}). We will just  develop the necessary background to state and discuss Cannon's conjecture that served as one of our motivations for  studying expanding Thurston maps. 

Let $G$ be a finitely generated group, and $S$ a finite set of
generators of $G$ that is {\em symmetric}, i.e., if $s$ is in
$S$, then its inverse $s^{-1}$ is also in $S$.  The 
{\em Cayley graph}\index{Cayley graph}
$\mathcal{G}(G, S)$ of $G$ with respect to $S$ is now
defined as follows: 
the group elements are the vertices of $\mathcal{G}(G, S)$,
and one joins two vertices given by  $g,h\in G$ 
by an edge if there exists 
$s\in S$ such that $g=hs$ (here we use the common
convention that juxtaposition of group elements means their
composition in the group).  Since $S$ is symmetric, this ``edge
relation'' for vertices is also symmetric, and so we consider
edges as undirected. If we identify each edge in
$\mathcal{G}(G, S)$ with a closed interval of length $1$, then
$\mathcal{G}(G, S)$ becomes a cell complex, where singleton sets
consisting of group elements are the cells of dimension $0$ and
the edges are the cells of dimension $1$. The graph
$\mathcal{G}(G, S)$ is connected and carries a unique path metric
so that each edge is isometric to the unit interval $[0,1]$.  In
the following, we always consider $\mathcal{G}(G, S)$ as a metric
space equipped with this path metric. Then  $\mathcal{G}(G, S)$ is proper and geodesic.

The group $G$ is called {\em Gromov
  hyperbolic}\index{Gromov!hyperbolic!group}  
if the metric space $\G(G,S)$ is Gromov hyperbolic for some (finite and symmetric)
set  $S$ of generators of $G$.
If this is the case, then 
$\G(G,S')$ is Gromov hyperbolic for all  generating 
sets $S'$. This essentially follows from the fact that 
$\G(G,S)$ and $\G(G,S')$ are quasi-isometric.  

Examples of Gromov hyperbolic groups are free groups, fundamental groups of compact negatively curved manifolds, or 
small cancellation groups. 

If $G$ is a Gromov hyperbolic group, then one defines its boundary at
infinity as $\partial_\infty G=\partial_\infty \G(G,S)$. 
\index{d -X@$\geo G$} 
A priori this depends on the choice of the generating set $S$, but if $S'$ is another   generating set,  then there is a natural identification 
 $\partial_\infty \mathcal{G}(G,S')\cong\partial_\infty \mathcal{G}(G,S)$; namely, 
 since $G$ is cobounded in $\G(G,S)$  and $\G(G,S')$,  one can represent points in the  boundaries of both spaces by equivalence classes of sequences  in $G$ converging to infinity where the equivalence relation  is independent of the generating set.  So $\partial_\infty G$ is well-defined. 
 
One has to be careful though when one considers visual metrics. If $\varrho$ is   a visual metric on $\partial_\infty \mathcal{G}(G,S)$, then in general $\varrho$ will not be a  visual metric  on $\partial_\infty \mathcal{G}(G,S')$; but if $\varrho'$ is a visual metric on $\partial_\infty \mathcal{G}(G,S')$, then 
 $\varrho$ and $\varrho'$ are  quasisymmetrically equivalent. In other words,   
the identity map between $(\partial_\infty  \mathcal{G}(G,S), \varrho)$ and $(\partial_\infty  \mathcal{G}(G,S'), \varrho')$ (given by the natural identification of these spaces as discussed)  is a quasisymmetry.  This follows from  Proposition~\ref{prop:qisomqs} and  the fact that $\mathcal{G}(G,S)$ and $\mathcal{G}(G,S')$
are quasi-isometric.   So $\partial_\infty  G$ carries a natural quasisymmetric  gauge. 

If we equip 
 $\partial_\infty  G$ with any of these visual metrics $\varrho$, then we can unambiguously speak of quasisymmetric  maps 
 on $\partial_\infty G$.  Another consequence of this is that  $\partial_\infty G$ carries a unique topology induced by any visual metric $\varrho$ on 
 $\partial_\infty G\cong \geo\mathcal{G}(G,S)$. 

 Letting a group element $g\in G$ act on the vertices 
of $\G(G,S)$ by left-translation, we get a natural action $G\acts
\G(G,S)$. This action is 
{\em geometric},\index{group action!geometric}\index{geometric group action}
i.e., it is isometric, properly discontinuous, and cocompact. To get a better understanding of the properties of a group, one often wants to find a ``better'' space 
than $ \G(G,S)$ on which $G$ admits a geometric action
(for a systematic exploration of this point of view see \cite{Kl}).

A related   question is how the topological structure of the boundary $\geo G$ of a Gromov hyperbolic group determines its algebraic structure.  Since the Cayley graph of $G$ and any of its subgroups of finite index are quasi-isometric and hence indistinguishable from the perspective 
of coarse geometry, one is mostly interested in {\em virtual properties} 
of $G$, i.e., algebraic properties that are true for some subgroup 
of finite index. 

The spaces $\geo G$ form a very restricted class.  For example, if 
$G$ is {\em non-elementary} (meaning that $\#\geo G\ge 3)$ and 
has topological dimension $0$, then $\geo G$ is homeomorphic to a Cantor set. Moreover, in this case $G$ is {\em virtually isomorphic} 
to a free group (i.e., some finite-index subgroup of $G$ is free). 

If  $\geo G$  is homeomorphic to a circle, then $G$ is {\em virtually 
Fuchsian}; so  $G$ is virtually isomorphic to a fundamental group of a compact hyperbolic surface, or equivalently, there is a geometric  action of 
$G$ on hyperbolic $2$-space $\Halb^2$ (see \cite{KB} for an overview on this subject).

If  $\geo G$  has no local cut points and has topological dimension one, then $\geo G$ is a Menger curve or a 
Sierpi\'nski carpet. Moreover, a conjecture due to Kapovich and Kleiner \cite{KK} predicts that in the latter case, $G$ admits a geometric action on a convex subset of $\Halb^3$ with non-empty totally geodesic boundary.   

This conjecture is related to (and implied by) another conjecture, 
due to Cannon (see \cite[p.~232]{Ca94}).

\begin{conj*}
[Cannon's conjecture.\ Version I]
\index{Cannon's conjecture|textbf}
Let $G$ be a Gromov hyperbolic group  and suppose $\geo G$ is homeomorphic to 
$\CDach$. Then there exists a geometric  action of $G$ on hyperbolic $3$-space $\Halb^3$. 
\end{conj*}

If this were true, then $G$ would be virtually isomorphic to the fundamental group of a compact hyperbolic $3$-manifold. 

In higher dimensions a corresponding statement is false; 
there are Gromov hyperbolic groups $G$ with $\geo G$ homeomorphic to an 
$n$-sphere, $n\ge 3$, that do not admit geometric actions on $\Halb^{n+1}$. One can obtain such examples as fundamental groups
of {\em Gromov-Thurston manifolds} \cite{GT}. These are  negatively-curved  closed manifolds that do not carry a hyperbolic metric, and exist in dimension $n+1\ge 4$.

Cannon's conjecture can be reformulated in equivalent form as a quasisymmetric uniformization problem (see \cite{Bo} for more discussion). 

\begin{conj*}[Cannon's conjecture.\ Version II]
  \index{Cannon's conjecture} 
  Let $G$ be a Gromov hyperbolic
  group and suppose $\geo G$ is homeomorphic to $\CDach$. Then
  $\geo G$ equipped with a visual metric is quasisymmetrically
  equivalent to  $(\CDach,\sigma)$.
\end{conj*}

In view of this formulation of the conjecture it is very
interesting to study the quasisymmetric uniformization problem
for metric $2$-spheres in general and ask for general conditions
under which such a sphere is quasisymmetrically equivalent to
the standard $2$-sphere $(\CDach,\sigma)$. Here the conditions should be
similar to those that one can establish for boundaries of Gromov
hyperbolic groups.

In all known examples where $\partial_\infty G$ is a
$2$-sphere, $G$ is (essentially) the fundamental group of a
hyperbolic manifold and $\geo G$ can naturally be identified
with $\geo \Halb^3$ which is the standard $2$-sphere. So in these
cases, no uniformization problem arises. Cannon's conjecture
predicts that there are no other examples.  One of the
difficulties in making progress on Cannon's conjecture is this
lack of non-trivial examples that may guide the intuition.

In contrast, the theory of Thurston maps provides 
a large class of self-similar fractal $2$-spheres that sometimes are,
and sometimes are not, quasisymmetrically equivalent to the standard $2$-sphere. By analyzing these examples, one may 
hope to discover some general features  that could  be relevant for  the solution of Cannon's conjecture.

\section{Quasispheres}
\label{sec:snowballs}

A metric space  quasisymmetrically equivalent to
 $(\CDach, \sigma)$  is called a
\emph{quasisphere}.\index{quasisphere} In view of the previous discussion of Cannon's conjecture and the characterization of rational Thurston maps as given  by Theorem~\ref{thm:S2vsf}~\ref{item:S2qsphere}, we now want to discuss two examples that may guide the  reader's intuition. As this is our main purpose here, we will skip the justification of most details. 

\begin{ex}  \label{ex:snowballs-1}

A \emph{snowball}\index{snowball} is a compact set in $\R^3$ constructed  in a similar  way
as  the set in the plane bounded by the classical von Koch snowflake
curve. The 
boundary of a snowball is a \emph{snowsphere}\index{snowsphere}
$\mathcal{S}$.  In many cases this is a quasisphere. 

The 
simplest example is 
obtained as follows (the  general construction can be found in 
\cite{Me08}). We start with the unit cube $[0,1]^3\subset
\R^3$ as the $0$-th approximation $\mathcal{B}^0$  of the
snowball. The boundary of  $\mathcal{B}^0$ is a polyhedral surface $\mathcal{S}^0$ consisting of six copies of the unit square $[0,1]^2\sub \R^2$ as faces. 
We divide each of these  six faces   into $5\times 5$ squares of side length
$1/5$ (or \emph{$1/5$-squares}).  On the 
$1/5$-square in the middle of each face we place a    cube  that  has  side length $1/5$ and 
sticks out of $\mathcal{B}^0$.     This results in a set $\mathcal{B}^1\supset \mathcal{B}^0$. The boundary of  $\mathcal{B}^1$ is a polyhedral surface 
$\mathcal{S}^1$
consisting of $6\times 29$ $1/5$-squares.  The 
procedure is now iterated;  namely,  each $1/5$-square is divided into
$5\times 5$ squares of side length $1/25$, on each middle square we
put a cube of side length $1/25$, and so on. We obtain an increasing sequence
$\mathcal{B}^0\sub \mathcal{B}^1\sub \dots $ of compact sets in $\R^3$. Their  union is
the snowball $\mathcal{B}$ with the snowsphere $\mathcal{S}\coloneqq \partial \mathcal{B}$ as its boundary.
One can show that   $\mathcal{S}$ is  indeed a  $2$-sphere. 
For each $n\in \N_0$ the boundary of  $\mathcal{B}^n$ is a polyhedral surface $\mathcal{S}^n$ that consists 
of $1/5^n$-squares.  
 The  surface $\mathcal{S}^n$ gives an   approximation of the
 snowsphere $\mathcal{S}$ 
that becomes increasingly better 
as $n\to \infty$.

One can also give another  construction of   $\mathcal{S}$
by a replacement procedure very similar to the one
in Section~\ref{sec:int-frac-sph}. For this we let the {\em generator} of the
snowball be  the polyhedral surface shown in
Figure~\ref{fig:gen_snow1}. The approximation $\mathcal{S}^{n+1}$ is then obtained from 
$\mathcal{S}^{n}$
 by replacing each $5^{-n}$-square of $\mathcal{S}^n$ with  a
scaled copy of the generator. If $X^n$ is one of the  $5^{-n}$-squares from which
$\mathcal{S}^n$ is built, and $X^{n+1}$ is a $5^{-(n+1)}$-square in the
scaled copy of the generator that replaces $X^n$,  we write $X^{n}
\supsim X^{n+1}$. 

\begin{figure}
  \centering
  \includegraphics{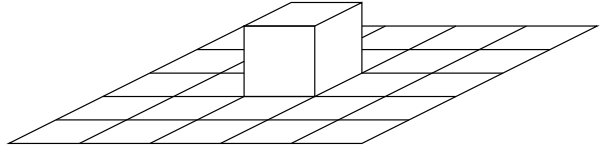}
  \caption{The generator of the snowsphere $\mathcal{S}$.}
  \label{fig:gen_snow1}
\end{figure}

The snowsphere $\mathcal{S}$ 
  inherits the Euclidean metric from $\R^3$. It is not hard to 
see that if  $x,y\in \mathcal{S}$,  then there exists a  rectifiable
path  $\gamma\subset \mathcal{S}$ joining $x$ and $y$ whose length is comparable 
to $|x-y|$. If we define $d(x,y)$ for $x,y\in \mathcal{S}$ as the infimum of the lengths of such paths, then we  get a length metric $d$ on $\mathcal{S}$  that  is bi-Lipschitz equivalent to the Euclidean metric on $\mathcal{S}\sub \R^3$ (see
\cite{Me02} and \cite{Me08}). 

Similarly to Section~\ref{sec:int-frac-sph}, we can estimate the
Euclidean metric on $\mathcal{S}$ in an intrinsic way. For this
we note that for each 
point $x\in \mathcal{S}$ there exist sequences
$X^0\supsim X^1\supsim \dots $ of $1/5^n$-squares $X^n$ such that
$\{ X^n\}$ Hausdorff converges to $\{x\}$ in $\R^3$.  Now for
$x,y\in \mathcal{S}$, $x\ne y$, we define (compare with
\eqref{eq:intromxy})
\begin{align}
  \label{eq:mdef_first}
  m(x,y): =\inf  \min\{ n\in \N_0 : X^n \cap Y^n=\emptyset \},
\end{align}
where the infimum is taken over all such  sequences $\{X^n\}$ for $x$ and $\{Y^n\}$ for $y$. 
Then  \begin{equation}
  \label{eq:visual_snow}
  \abs{x-y} \asymp 5^{-m(x,y)},
\end{equation}
where $C(\asymp)$ is independent of $x$ and $y$. So up to a multiplicative constant, the  Euclidean metric on $\mathcal{S}$ can  be recovered from the  combinatorics of the sets $X^n$. 

It is a small step from here to the theory of Gromov hyperbolic spaces. We   construct a graph $\G$ as follows (see Chapter~\ref{cha:Gromov} for very similar considerations). The set of
vertices of $\G$ is the set of   $1/5^n$-squares $X^n$ for all $n\in
\N_0$. It is  convenient to add another vertex $X^{-1}$.   We
declare that $X^{-1} \supsim X^0  $ for any $1/5^0$-square $X^0$. Then  each vertex 
of $\G$ as represented by $X^n$ has an attached level $n\in \N_0\cup \{-1\}$.

The set
of (undirected) edges of $\G$ is now given as follows. We connect two  distinct vertices by an edge if they have the same level $n$ and are represented by two $1/5^n$-squares $X^n$ and $Y^n$ with 
$X^n\cap Y^n \neq \emptyset$. Moreover, we join two vertices represented  by 
$X^n$ and $X^{n+1}$ if 
  $X^{n}\supsim X^{n+1}$. 
There are no other edges in $\G$. 

If we identify each edge with a copy of the unit interval $[0,1]$, then 
$\G$ carries a natural path metric (corresponding to combinatorial distance in 
$\mathcal{G}$ on the set of vertices). 
It can be shown that with this path metric $\G$ is a
Gromov hyperbolic metric space.  

There is a natural identification of $\mathcal{S}$ with the boundary at
 infinity $\geo \mathcal{G}$. Namely, if $x\in \mathcal{S}$, then we choose a sequence
$X^0\supsim X^1\supsim  \dots $ of $1/5^n$-squares  such that  $\{ X^n\}$ 
Hausdorff converges to $\{x\}$ in $\R^3$. Then  $\{ X^n\}$,  now considered as a sequence of vertices in 
$\mathcal{G}$,  converges to  infinity (see Section~\ref{sec:Grhyp}). Sending a point   $x$ to the equivalence class 
of $\{X^n\}$ (considered as a point on  $\geo \mathcal{G}$), we get  a bijection between $\mathcal{S}$ and $\geo \mathcal{G}$  that  we use to identify these two sets. 

We choose the basepoint $p=X^{-1}$ in $\mathcal{G}$.  Then for
the Gromov product 
$(x\cdot y)_p$\index{Gromov!product}
of two points $x,y\in \mathcal{S}=\geo \mathcal{G}$ we have 
\begin{equation*}
  m(x,y) - c \leq (x\cdot y)_p \le m(x,y) + c,  
\end{equation*}
where $c\ge 0$ is a constant independent of $x$ and $y$. 
{}From this and \eqref{eq:visual_snow} it follows that the Euclidean metric on $\mathcal{S}=\geo \mathcal{G}$
is a visual metric in the sense of Gromov hyperbolic spaces; indeed, it satisfies 
\eqref{vismetric} with 
$\Lambda=5$. 
 It is easy to see 
that there are no visual metrics with $\Lambda>5$. 

The snowsphere $\mathcal{S}$ (equipped with the metric inherited from $\R^3$) is
 a quasisphere. This was shown in \cite{Me02} (see also \cite{Me08}
and \cite{Me09a}). For the proof one  constructs  a  rational
Thurston map (in a non-obvious way)  leading to  sets that mirror  the combinatorics of the sets $X^n$ 
related to  $\mathcal{S}$. One can use this to  show  directly that $\mathcal{S}$
is a
quasisphere, 
 or   one invokes the general criterion given by Theorem~\ref{thm:S2vsf}~\ref{item:S2qsphere}. 
\end{ex}

\begin{ex}[A non-quasisphere]
\label{ex:non-quasisphere}
Our second example is the sphere $\mathcal{S}$ that was discussed
in Section~\ref{sec:int-frac-sph}. It is equipped with the metric
$\varrho$ defined in \eqref{eq:def_visual_1}.
As we already remarked, it follows 
from the fact that the associated Thurston map $h$ has a Thurston obstruction in combination with Theorem~\ref{thm:S2vsf}~\ref{item:S2qsphere} that $\mathcal{S}$ is not a quasiphere.
Here we want to outline a direct argument for this statement (it emerged in discussions with B.~Kleiner). We use the notation from 
Section~\ref{sec:int-frac-sph}.

 Consider the top part of $\mathcal{S}$. It is given by all equivalence classes of sequences 
 $\mathcal{X}^0\supsim \mathcal{X}^1 \supsim \dots$, where 
$\mathcal{X}^0$ is the top white $0$-tile of $\mathcal{S}^0$.  
{}From this top part we remove  all the ``flaps'' that  were successively added 
to $\mathcal{X}^0$ in the construction of $\mathcal{S}$. What remains is a subset $Z\sub \mathcal{S}$
that looks like the unit square $U=[0,1]^2$ with countably many slits
 (see Figure~\ref{fig:slit_carpet}; related to $Z$ are the ``slit carpets'' considered in  \cite{Mer10}). 
These slits are all parallel to one of the sides of $U$, say to $[0,1]\times \{0\}$. Each point 
$z\in Z$ corresponds to a unique point in $U$. This gives a surjective map $\pi\: Z\ra U$. 
If a point  $p\in U$ is an interior point of one of the slits, then there  are two points in $Z$ (one for each side of the slit) that map to $p$. For all other points $p\in U$ we have $\#\pi^{-1}(p)=1$.

\begin{figure}
 \centering
 \begin{overpic}
   [width=10cm, tics=20, 
   ]{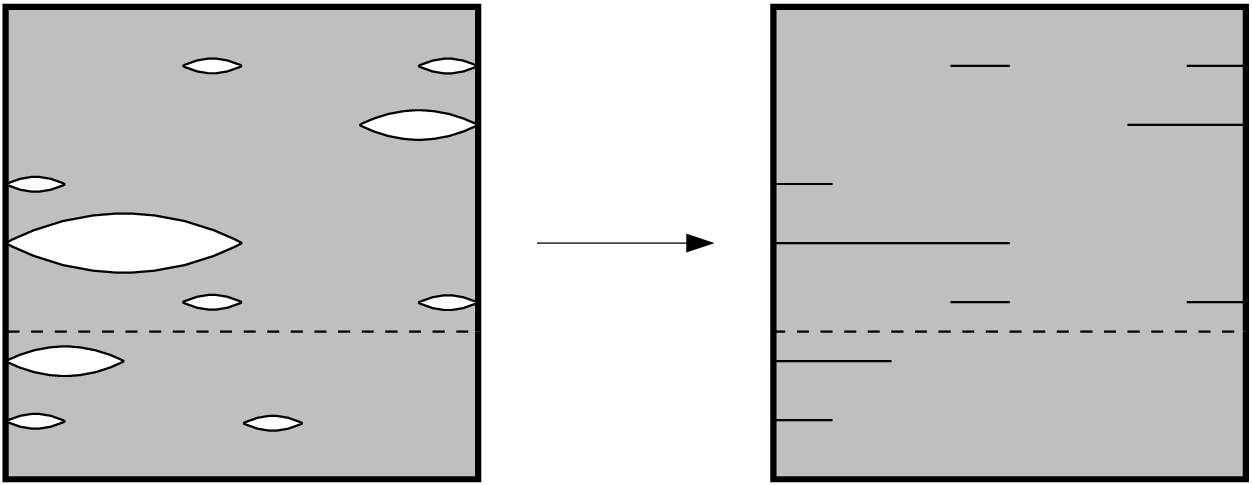}
   \put(48,21){$\pi$}
   \put(102,2){$U$}
   \put(40,2){$Z\subset \mathcal{S}$}
   \put(101,11){$\scriptstyle [0,1]\times\{y\}$}
 \end{overpic}
 \caption{The set $Z$.}
 \label{fig:slit_carpet}
\end{figure}
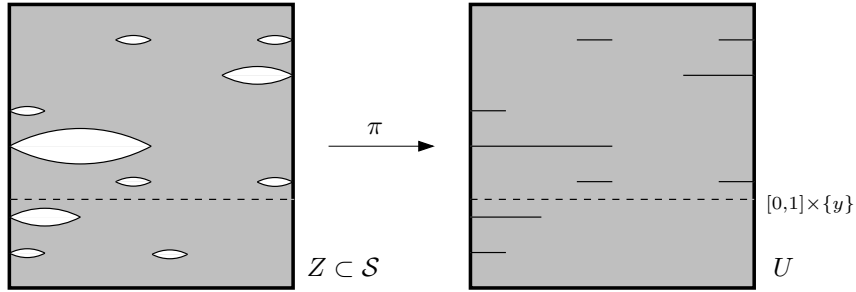

 We equip $Z$ with the metric $\varrho$ and $U$ with the Euclidean metric. Then the map $\pi\: 
 Z\ra U$ is Lipschitz. Actually, $\pi$ is {\em David-Semmes
   regular} (as defined in \cite[Chapter 12]{DS}). For $\pi$ this means 
that in addition to being  Lipschitz,    there exists a number $N\in \N$ such that the preimage $\pi^{-1}(B(p,r))$ 
 of each Euclidean  ball $B(p,r)$ in $U$ can be covered by $N$ balls in $Z$ of the same radius $r$. 
 
At least on an intuitive level, one can see that $\pi$ has this last property as follows. If no slit cuts 
through 
$B(p,r)$, then $\pi^{-1}(B(p,r))$
is contained in a ball in $Z$ whose radius is not much larger, and hence comparable to $r$. 
If a slit cuts through  $B(p,r)$, then $\pi^{-1}(B(p,r))$ is split into two parts each of which 
 is contained in a ball in $Z$  with radius  comparable to $r$. This implies that in  any case, 
 $\pi^{-1}(B(p,r))$ can be covered by a controlled number of  balls in $Z$ with radius $r$.

 Since $U$ is Ahlfors $2$-regular, and $\pi$ is David-Semmes regular, $Z$ is  Ahlfors $2$-regular
 as well (\cite[Lemma~12.5]{DS}). 
 
Let $\Gamma$ be the set of all paths in $Z$ that project under $\pi$ to a line segment that has the form  $[0,1]\times \{y\}$, $y\in [0,1]$,  and does not contain a slit. Restricted to a path $\ga\in \Gamma$,  the
 map $\pi$ is bi-Lipschitz with a uniform constant independent of $\ga$. This implies that 
 the $2$-modulus of $\Gamma$ in $Z$ (see \cite[Section~7.3]{He} for the definition of the modulus of a path family) cannot be much smaller 
 than the $2$-modulus of $\pi(\Gamma)$ and is hence positive. Together with the Ahlfors $2$-regularity of $Z$ this implies that an image of $Z$ under any quasisymmetric homeomorphism has Hausdorff
 dimension $\ge 2$ (\cite[Theorem~15.10]{He}). 
 
 One other property of $Z$ will be  important. Namely, $Z$ is a
 {\em porous}\index{porous} 
subset of $\mathcal{S}$. This 
 means that there exists a constant $c\in (0,1)$ with the following property: if $a\in Z$  and  $r>0$ with 
 $r\le \diam(Z)$ are arbitrary, then there exists $x\in B(a,r)$ with $B(x,cr)\cap Z=\emptyset$.  
 So the set $B(a,r)\cap Z$ has the  ``hole''  $B(x,cr)$ of comparable size. 
 
 The porosity of $Z$ follows from the fact that each ball $B(a,r)$ in $\mathcal{S}$ with $r\le \diam (\mathcal{S})$ contains a flap of size comparable to $r$. Since in the construction of $Z$ we removed all flaps from the top side of  $\mathcal{S}$, this means that all sufficiently small balls centered in $Z$ contain a hole of about the same size. 
 
 Now we can see that $\mathcal{S}$ is not a quasisphere as
 follows. We argue by contradiction and assume that there exists
 a quasisymmetry $\varphi$ of $\mathcal{S}$ onto $\CDach$
 (equipped with the chordal metric). From what we have discussed
 above, it then follows that the Hausdorff dimension of
 $\varphi(Z)$ is $\ge 2$.  On the other hand, images of porous
 sets under quasisymmetries are porous (see \cite[Theorem
 4.2]{Va1} for a proof in $\R^n$; it easily generalizes to a
 metric space setting). Hence $\varphi(Z)$ is a porous subset of
 $\CDach$. A porous subset of an Ahlfors $Q$-regular space,
 $Q>0$, has Hausdorff dimension $<Q$ (\cite[Lemma~5.8]{DS}). It
 follows that the Hausdorff dimension of $\varphi(Z)$ is
 $<2$. This is a contradiction, and so $\mathcal{S}$ cannot be a
 quasisphere.
\end{ex}

 \ifthenelse{\boolean{singlechapter}}{

%
%


\chapter{Cell decompositions}
\label{cha:celldecomp}

In this chapter we discuss some technical, but very crucial
aspects of our work, namely cell decompositions and their
relation to Thurston maps. 
Since this  is the basis of our approach to the investigation of Thurston maps, we collected  all the relevant facts in one place with a detailed  presentation.

Accordingly, this chapter is quite long and covers a mix of 
concepts and results that are fairly standard (Sections~\ref{s:celldecomp} and \ref{sec:2spherecd}),  are important right away for our further developments (Sections~\ref{sec:tiles},~\ref{sec:flowers}, and \ref{sec:opp}),
 or are not needed until much later (Sections~\ref{sec:labelings} and~\ref{sec:mapsfromdecomp} are not used before Chapter~\ref{cha:subdivisions}). 
For this reason, the reader may  not  want to peruse 
 this chapter in  linear order 
and could just skim 
through some of its parts, in  particular at first reading.
 
%

%
%

We start with a general review of cell decompositions for
arbitrary spaces (Section~\ref{s:celldecomp}). The reader
already acquainted with these concepts  
is encouraged to look
at
Definition~\ref{def:celldecomp} to get familiar with our
terminology and notation. We also define refinements of cell
decompositions and  cellular Markov partitions for a map
(see Definition~\ref{def:ref} and
Definition~\ref{def:cellular}). 
For the most part this is in preparation
of Chapter~\ref{cha:subdivisions}.

In Section~\ref{sec:2spherecd} we
specialize to cell decompositions of $2$-spheres. The most important facts about them are recorded in Lemma~\ref{lem:specprop}. 
 We also consider isomorphisms of cell complexes
(Definition~\ref{def:compiso}) and the homeomorphisms that are
induced by them (Lemma~\ref{lem:isocellhomeo}).
     
In Section~\ref{sec:tiles} we consider cell decompositions of a 
$2$-sphere $S^2$ induced by a Thurston map $f\:S^2\ra S^2$. For this we consider 
a Jordan
curve $\CC\sub S^2$   with $\post(f)\sub \CC$ and an associated cell decomposition
$\DD^0(f,\CC)$ of $S^2$. Its cells are given by the points  
 in $\post(f)$ as vertices, the closed arcs into which the points in $\post(f)$
divide $\CC$, and the closures of the two components of
$S^2\setminus \CC$. Pulling this decomposition back by $f^n$, we
obtain a cell decomposition $\DD^n=\DD^n(f,\CC)$ of $S^2$ for each $n\in \N_0$ (see Corollary~\ref{cor:pullback_D0} and
Definition~\ref{def:DDn}). These cell decompositions $\DD^n$  are our most important tool
for studying Thurston maps.  The properties of
 $\DD^n$ are collected in Proposition~\ref{prop:celldecomp}. This is an elementary, but 
   central result of this chapter and will be  used throughout this work. 

The  $2$-dimensional {cells} or {tiles}
in $\DD^0(f,\CC)$ are the  two closed Jordan regions in $S^2$ bounded by
$\CC$. If we assign the colors black and white to them, then we can pull this coloring back by $f^n$ and obtain 
colors for the tiles in $\DD^n$ (see
Lemma~\ref{lem:colortiles}). This is closely  related to  the more general notion of a
\emph{labeling} considered in Section~\ref{sec:labelings}. For the cell decompositions 
$\DD^0=\DD^0(f,\CC)$ and $\DD^1=\DD^1(f,\CC)$ as in  Definition~\ref{def:DDn}, this is simply the map
$L\colon \DD^1\to \DD^0$ given by $L(c) =f(c)$ for 
$c\in \DD^1$. We will turn this around in
Section~\ref{sec:mapsfromdecomp}, and construct Thurston maps
from 
 cell decompositions $\DD^1$ and $\DD^0$ and such a labeling
$L$ (see Proposition~\ref{prop:thurstonex}). These notions will be
revisited later in Chapter~\ref{cha:subdivisions}, where we will assume
in addition  that $\DD^1$ is a refinement of $\DD^0$. This will
allow us to define and describe  Thurston maps by finite
combinatorial data. The reader may safely skip 
Sections~\ref{sec:labelings} and~\ref{sec:mapsfromdecomp}  until this material is needed in 
Chapter~\ref{cha:subdivisions}.

In Section~\ref{sec:flowers} we consider \emph{flowers}. These
are simply connected neighborhoods of vertices in $\DD^n=\DD^n(f,\CC)$ (see
Definition~\ref{def:flower}). Their properties are summarized in Lemma~\ref{lem:flowerprop}. They behave well under the maps 
$f^k$ (see Lemma~\ref{lem:mapflowers}). We also 
introduce \emph{edge-flowers} as  neighborhoods of edges in
 $\DD^n$ (see Definition~\ref{def:edgeflower}
and Lemma~\ref{lem:edgeflower}). 

In the last section (Section~\ref{sec:opp}) we introduce the notion of
\emph{joining opposite sides} of the Jordan curve $\CC$. The most important result here is
Lemma~\ref{lem:maptotop}, which says that if a connected set $K$ 
joins two disjoint cells in $\DD^n(f,\CC)$, then $f^n(K)$ joins opposite
sides of $\CC$. This fact  will be of major
importance when we construct visual metrics in
Chapter~\ref{cha:visual-metrics}. 

\section{Cell decompositions in general}\label{s:celldecomp}

Here   we  review some facts about cell decompositions of arbitrary spaces. Most  of this material is well known
(see, for example, \cite[Chapter~1]{CF67}).  
  For the purpose of the  present work we could have restricted ourselves to cell decompositions of subsets of a  $2$-sphere,   but it is more transparent to discuss the topic  in greater generality.

In this section $\mathcal{X}$ will always be a locally compact Hausdorff space.  A {\em 
(compact topological) cell $c$ of dimension
$n=\dim(c)\in \N$}\index{cell}\index{cell!dimension of}\index{dimension of cell} 
in $\mathcal{X}$ is a set $c\sub \mathcal{X}$ that  is   homeomorphic to the closed unit ball $\overline \B^n$  in $\R^n$.  We denote by $\partial c$ the set of 
points corresponding to $\partial \overline \B^n$ under such a homeomorphism between $c$ and $\overline \B^n$. This is  independent of the homeomorphism chosen, and the   set  $\partial c$ is  well-defined. We call  $\partial  c$ the   {\em boundary}  and  $\inte(c)=c\setminus \partial c$ the {\em  interior}   of $c$.  Note that   boundary and  interior 
of $c$ in this sense will in general not agree with 
the boundary and interior  of $c$  regarded as a subset of the topological space $\mathcal{X}$. 
A {\em cell of dimension $0$} in  $\mathcal{X}$ is a set $c\sub \mathcal{X}$ consisting of a single point. We set $\partial c=\emptyset$ and $\inte(c)=c$ in this case. 

\begin{definition}[Cell decompositions]\label{def:celldecomp}
Suppose  that $\DD$ is a  collection of cells in a locally compact Hausdorff space $\mathcal{X}$.  We say that $\DD$ is a  
{\em cell decomposition}\index{cell!decomposition} of $\mathcal{X}$ 
provided the following conditions are satisfied: 

\begin{enumerate}
  
\item
\label{item:def_cell1}
  The union of all cells in $\DD$ is equal to $\mathcal{X}$.

\item
\label{item:def_cell2}
  We have $\inte(\sigma)\cap \inte(\tau)= \emptyset$,  whenever $\sigma,\tau\in \DD$, $\sigma\ne \tau$.

\item
\label{item:def_cell3}
  If $\tau\in \DD$, then $\partial \tau$ is a union of cells in $\DD$.

\item
\label{item:def_cell4}
  Every point  in $\mathcal{X}$ has a neighborhood that meets only finitely many cells in $\DD$. 

\end{enumerate}
\end{definition}

Note that in the literature one often uses the term {\em regular} 
for cell decompositions as in 
Definition~\ref{def:celldecomp} in order to distinguish them from more 
general notions of cell decompositions (as in the theory of CW-complexes, for example).

If $\DD$ is a collection of cells in some ambient space
$\mathcal{X}$, then we call $\DD$ a 
{\em cell complex}\index{cell!complex} if $\DD$ is a cell decomposition of the underlying set 
$$|\DD|\coloneqq \bigcup\{c:c\in \DD\}. $$

Suppose $\DD$ is a cell decomposition of $\mathcal{X}$. 
By \ref{item:def_cell4}, every compact subset of $\mathcal{X}$ can only meet finitely
many cells in  $\DD$.  In particular, if $\mathcal{X}$ is compact, then $\DD$ consists of only finitely many cells.  Moreover,  for each $\tau \in \DD$, the set  
$\partial \tau$ is compact and hence equal  to a finite union of cells in $\DD$. 
It follows from basic  dimension theory that if 
\ $\dim(\tau)=n$, 
then $\partial \tau$ is equal to  a  union of cells in $\DD$ that have dimension $n-1$. 
 
The 
union  $\mathcal{X}^n$ of all cells in $\DD$ of dimension $\le n$ is called the
$n$-{\em skeleton}\index{n9@$n$-!skeleton}\index{skeleton} of the cell decomposition.
It is useful to set $\mathcal{X}^{-1}=\emptyset$.  
It follows from property
\ref{item:def_cell4} of
a cell decomposition that 
$\mathcal{X}^n$ is a closed subset of $\mathcal{X}$ for each $n\in \N_0$. 
By the last remark in the previous paragraph, we have  $\partial \tau\sub \mathcal{X}^{n-1}$  for each $\tau \in \DD$ with $\dim(\tau)=n$.

\begin{lemma} \label{lem:uniondisjint}
 Let $\DD$ be a cell decomposition of  $\mathcal{X}$. Then for each $n\in \N_0$ the $n$-skeleton $\mathcal{X}^n$ is equal to the disjoint union of the sets $\inte(c)$, $c\in \DD$, $\dim(c)\le n$.  
The space $\mathcal{X}$ is equal to the disjoint union of the sets
$\inte(c)$, $c\in \DD$. Similarly, every cell $\tau\in \DD$ is the
disjoint union of the sets $\inte(c)$, where $c\in \DD$ and
$c\subset\tau$. 
\end{lemma} 

So in particular, the interiors of the cells in a cell decomposition  partition the space $\mathcal{X}$. This is of prime importance and will be used frequently throughout this work.  

\begin{proof} We show the first statement by induction on $n\in \N_0$. 
Since $\inte(c)=c$ for each cell $c$ in $\DD$ of dimension $0$, it is clear that $\mathcal{X}^0$ is the disjoint union of the interiors of all cells $c\in \DD$ with $\dim(c)=0$. 

Suppose that the first statement is true for $\mathcal{X}^n$, and let $p\in
\mathcal{X}^{n+1}$ be arbitrary.  If $p\in \mathcal{X}^n$, then $p$ is
contained in the interior of a cell $c\in \DD$ with $\dim(c)\le n$ by
induction hypothesis. In the other case, $p\in
\mathcal{X}^{n+1}\setminus \mathcal{X}^{n}$, and so there exists $c\in
\DD$ with $\dim(c)=n+1$ and $p\in c$. Since $\partial c\sub
\mathcal{X}^n$, it follows that $p\in c\setminus \partial
c=\inte(c)$. So $\mathcal{X}^{n+1}$ is the union of the interiors of
all cells $c$ in $\DD$ with $\dim(c)\le n+1$. This union is disjoint,
because distinct cells in a cell decomposition have disjoint
interiors.

The second statement follows from the first, and the obvious fact that
$\mathcal{X}=\bigcup_{n\in \N_0}\mathcal{X}^n$.    

To see the last statement, let $\DD_\tau\coloneqq \{c \in \DD :
c\subset \tau\}$. Then it is clear  that $\DD_\tau$ is a cell
decomposition of (the compact Hausdorff space) $\tau$. So the claim follows from the  previous
statement. 
\end{proof}

The lemma implies 
that if $\tau\in \DD$ and $\dim(\tau)=n$, 
 then  each point $p\in \inte(\tau)$ is an interior point of
 $\tau$  regarded as a subset 
of the topological space $\mathcal{X}^n$. 
 Indeed, we can choose a neighborhood $U$ of $p$ such that 
  $U\cap \sigma=\emptyset$ whenever $\sigma\in \DD$ and $p \not\in \sigma$. Then $U\cap \mathcal{X}^{n-1}=\emptyset $ and so $U\cap \mathcal{X}^n\sub \inte(\tau)$ as follows from the lemma;  hence $p$ is an interior point of $\inte(\tau)$ in $\mathcal{X}^n$.

\begin{lemma}\label{lem:celldecompint} Let $\DD$ be a cell decomposition of  $\mathcal{X}$.

\begin{enumerate}
\item
\label{item:cell_decomp1}
If  $\sigma$ and $\tau$ are  two distinct cells in $\DD$  with $\sigma\cap \tau\ne \emptyset$, then one of the following statements holds:  $\sigma\sub \partial \tau$, $\tau \sub \partial \sigma$, or 
$\sigma\cap \tau =\partial \sigma \cap \partial\tau$  and this intersection 
 consists of cells in $\DD$ of   dimension strictly less than $\min\{\dim(\sigma), \dim(\tau)\}$.

\item
  \label{item:cell_decomp2}
  If $\sigma,\tau_1, \dots, \tau_n$ are cells in $\DD$ and $\inte(\sigma)\cap (\tau_1\cup \dots \cup \tau_n)\ne \emptyset$, then $\sigma\sub \tau_i$ for some $i\in\{1,\dots, n\}$. 
\end{enumerate}  
\end{lemma}

\begin{proof} $\,\!$
\ref{item:cell_decomp1} 
  We may assume that $l\coloneqq \dim(\sigma)\le
m\coloneqq \dim (\tau)$, and prove the statement by induction on $m$. 
The case $m=0$ is vacuous and hence trivial.  Assume that the statement is true  whenever both cells have dimension   $<m$. 
If $l=m$ then by definition of a cell decomposition $\inte(\sigma)$ is disjoint from $\tau \sub \inte(\tau) \cup \mathcal{X}^{m-1}$, and similarly 
$\inte (\tau)\cap \sigma=\emptyset$. Hence 
$\sigma\cap \tau = \partial \sigma \cap \partial \tau $. Moreover, both sets $\partial \sigma$ and $\partial \tau$ consist of finitely many cells in $\DD$ of dimension $\le m-1$. Applying the 
induction hypothesis to pairs of these cells, we see that 
$\partial \sigma  \cap  \partial \tau$ consists of cells of dimension $<m$ as desired. 

If $l<m$, then $\sigma \sub \mathcal{X}^{m-1}$ and so $\sigma \cap \inte (\tau)=\emptyset$. This shows that $\sigma\cap \tau=\sigma\cap \partial \tau$. Moreover, we have   $\partial \tau=c_1\cup \dots \cup c_s$, where $c_1, \dots , c_s$ are cells of dimension $m-1$.  
So we can apply the induction hypothesis to the pairs 
$(\sigma, c_i)$.  If $\sigma=c_i$ or $\sigma \sub \partial c_i$ for some $i$, then $\sigma \sub \partial \tau$; we cannot have $c_i\sub \partial \sigma$, because  $c_i$ has dimension $m-1$, and $\partial \sigma$ is a set of topological  dimension $<m-1$. So  if none of the first  possibilities occurs, then  $\sigma \cap c_i=\emptyset $,  or $\sigma \cap c_i=\partial \sigma\cap\partial c_i$ and this set consists of 
cells of dimension $<l$ (by induction hypothesis) contained in
 $\partial c_i\sub c_i\sub \partial \tau$ for all $i$. In this  case $\sigma\cap \tau=\partial \sigma \cap \partial \tau$, and this set  consists of cells of dimension $<l$ as desired. The claim follows. 

\smallskip
\ref{item:cell_decomp2}
There exists $i\in \{1,\dots, n\}$ with $\inte(\sigma)\cap \tau_i\ne \emptyset$. By the alternatives in \ref{item:cell_decomp1} we then must have $\sigma=\tau_i$ or $\sigma\sub \partial\tau_i$. Hence 
$\sigma\sub \tau_i$. 
\end{proof}

\begin{lemma}\label{lem:conncomp} Let  $A\sub \mathcal{X}$  be a   closed set, and $U\sub \mathcal{X}\setminus A$ be a non-empty open and connected  set. If $\partial U\sub A$, then $U$ is a connected component of $\mathcal{X}\setminus A$. 
\end{lemma}

\begin{proof} Since $U$ is a non-empty connected set in the complement of $A$, this set is contained in a unique connected component $V$ of $\mathcal{X}\setminus A$. 
Since $\partial U\sub A\sub \mathcal{X}\setminus V$, we have $V \cap \overline U=V\cap U=U$ showing that $U$ is relatively open and closed in $V$.
Since $U\ne \emptyset$ and $V$ is connected, it follows that $U=V$ as desired. \end{proof}

\begin{lemma} \label{lem:opencells}
Let  $\DD$ be a cell decomposition of $\mathcal{X}$ with
$n$-skeleton $\mathcal{X}^n$, where $n\in \{-1\}\cup\N_0$. 
Then for each $n\in \N_0$ the non-empty connected components of $\mathcal{X}^{n}\setminus \mathcal{X}^{n-1}$ are precisely the sets $\inte(\tau)$, where $\tau\in \DD$ and  $\dim(\tau)=n$.\end{lemma} 

\begin{proof} Let $\tau$ be a cell in $\DD$ with $\dim(\tau)=n$. 
Then  
$\inte(\tau)$ is a connected set contained in $\mathcal{X}^n\setminus \mathcal{X}^{n-1}$ that is relatively open with respect to $\mathcal{X}^n$. Its relative boundary is a subset of $\partial \tau$ and hence contained in the closed set $\mathcal{X}^{n-1}$. It follows by Lemma~\ref{lem:conncomp} that  $\inte(\tau)$ is equal  to a   component $V$ of $\mathcal{X}^n\setminus \mathcal{X}^{n-1}$.

Conversely, suppose that $V$ is a non-empty  connected component of
$\mathcal{X}^n\setminus \mathcal{X}^{n-1}$. Pick a point $p\in V$. Then $p$ lies 
in the interior of a unique cell $\tau\in\DD$ with $\dim(\tau)=n$. It follows from the first part of the proof that  $V=\inte(\tau)$. 
\end{proof}

\begin{definition}[Refinements]\label{def:ref}
Let $\DD'$ and $\DD$ be two cell decompositions of the  space $\mathcal{X}$. We
say that $\DD'$ is a {\em
  refinement}\index{cell!decomposition!refinement}\index{refinement of
  cell decomposition} of $\DD$ if
the following two conditions are satisfied:
\begin{enumerate}

\item
\label{item:ref_1}
  For every cell $\sigma\in \DD'$ there exists  a cell 
$\tau\in \DD$ with $\sigma\sub \tau$.

\item
\label{item:ref_2}
  Every cell $\tau\in \DD$ is the union of all cells 
$\sigma\in \DD'$ with $\sigma\sub \tau$. 

\end{enumerate}
\end{definition}

It is easy to see that if $\DD'$ is a refinement  of $\DD$ and  $\tau \in \DD$, then the cells $\sigma\in \DD'$ with $\sigma\sub \tau$ form a cell decomposition of $\tau$. Moreover, every cell $\sigma\in \DD'$ arises in this way from some $\tau \in \DD$. So roughly speaking, the  refinement  $\DD'$ of the  cell decomposition $\DD$  is obtained by decomposing each cell in $\DD$  into smaller cells. We informally refer to this process  as {\em subdividing} the cells in $\DD$ by  the smaller cells in $\DD'$.

\begin{lemma}\label{lem:mincell} Let $\DD'$ and $\DD$ be two cell decompositions of  $\mathcal{X}$, and $\DD'$ be a refinement of $\DD$. 
Then for every cell $\sigma\in \DD'$ there exists a minimal 
cell $\tau \in \DD$ with $\sigma\sub \tau$, i.e., if $\widetilde\tau\in \DD$ is another  cell with $\sigma\sub \widetilde\tau$, then $\tau \sub \widetilde\tau$. Moreover, $\tau$ is the unique cell in $\DD$ with $\inte(\sigma)\sub \inte (\tau)$. 
\end{lemma}

\begin{proof} First note that if $\sigma\in \DD'$, $\tau_1, \dots, \tau_n\in \DD$ and 
$$\inte(\sigma)\cap( \tau_1\cup \dots \cup \tau_n)\ne \emptyset, $$ then  $\sigma\sub \tau_i$ for some $i\in \{1, \dots, n\}$. Indeed, by  definition of a refinement  the union of all cells in  $\DD'$ contained in some $\tau_i$  covers $\tau_1\cup \dots \cup \tau_n$. Hence this union  meets $\inte(\sigma)$. It follows from Lemma~\ref{lem:celldecompint}~\ref{item:cell_decomp2} that $\sigma$ is contained in one of these cells from $\DD'$ and hence in one of the cells $\tau_i$. 

Now if $\sigma\in \DD'$ is arbitrary, then $\sigma$ is contained in some cell of  $\DD$ by definition of a refinement, and hence in a cell $\tau\in \DD$ of minimal dimension.
Then $\tau$ is minimal among all cells in $\DD$ containing $\sigma$. Indeed, let $\widetilde\tau\ne \tau$ be another cell in $\DD$ containing $\sigma$. We want to show that $\tau\sub\widetilde \tau$.
 
   One of the alternatives in Lemma~\ref{lem:celldecompint}~\ref{item:cell_decomp1} occurs for $\tau$ and $\widetilde \tau$. 
   If
$\tau\subset \partial \widetilde\tau \subset \widetilde\tau$ we are done. The second
alternative, $\widetilde\tau\subset \partial \tau$,  is impossible, since $\tau $ has  minimal dimension
among all cells containing $\sigma$. 
 The third alternative leads to  $\sigma\sub \tau\cap\widetilde\tau=\partial \tau\cap \partial \widetilde\tau$, where the latter intersection 
consists of cells in $\DD$ of dimension 
$<\dim(\tau)$. By the first part of the proof $\sigma$ is contained in one of these cells, again contradicting the definition of $\tau$. Hence $\tau$ is minimal. 

We have $\inte(\sigma)\sub \inte(\tau)$; for otherwise $\inte(\sigma)$ meets $\partial \tau$ which is a union of cells in $\DD$. Then $\sigma$ would be contained in one of these cells by the first part of the proof. This contradicts the minimality of $\tau$.

Finally, it is clear that $\tau\in \DD$ is the unique cell with 
$\inte(\sigma)\sub\inte( \tau)$, because distinct cells in a cell decomposition have disjoint interiors. 
\end{proof}

\begin{definition}[Cellular maps and cellular Markov partitions]\label{def:cellular}
Let $\DD'$ and $\DD$ be two cell decompositions of  $\mathcal{X}$, and $f\: \mathcal{X}\ra
\mathcal{X}$  be a continuous map. We say that $f$ is  {\em
  cellular}\index{cellular!map}\index{map!cellular} for  $(\DD', \DD)$ if the following  condition is   satisfied: if  $\sigma\in \DD'$ is arbitrary, then $f(\sigma)$ is a
  cell in $\DD$ and 
$f|\sigma$ is a homeomorphism of $\sigma$ onto $f(\sigma)$.

If $f$ is cellular for $(\DD',\DD)$ and $\DD'$ is a refinement of $\DD$, then the pair $(\DD', \DD)$ is  called a {\em cellular Markov partition}\index{cellular!Markov partition} for $f$. 
\end{definition} 

 Cellular Markov
partitions will become important only later starting in Chapter~\ref{cha:subdivisions}. Since
almost all of our  examples of Thurston maps  are
in fact constructed from cellular Markov partitions, we chose to
introduce this notion already here.

\section{Cell decompositions of $2$-spheres}
\label{sec:2spherecd}

 We now turn to   cell decompositions 
of $2$-spheres\index{cell!decomposition! 
of $2$-sphere}. We first review some standard concepts and  results from plane topology (see \cite{Moi} for general background and more details).

Let $S^2$ be a $2$-sphere. An 
{\em arc}\index{arc}
$\alpha$ in $S^2$ is  a
homeomorphic image of the  unit interval $[0,1]$. The points
corresponding to $0$ and $1$ under such a homeomorphism
are called the {\em endpoints} of $\alpha$. They are the unique points
$p\in \alpha$ such that $\alpha\setminus \{p\}$ is connected.  If $p$
is an {\em interior point} of $\alpha$, i.e., a point in $\alpha$
distinct from the endpoints, then there exist arbitrarily small connected open
neighborhoods $W\sub S^2$ of $p$ such that  
$W\setminus \alpha$ has precisely two open connected components $U$
and $V$.  
 
A 
{\em closed Jordan region}\index{Jordan region} 
$X$ in $S^2$ is      
a homeomorphic image of the closed unit disk $\overline \D$. The
boundary $\partial X$ of a closed Jordan region $X\sub S^2$ is a
{\em Jordan curve},\index{Jordan curve}
i.e., the homeomorphic image of the unit circle $\partial \D$. 
 If $J\sub S^2$ is a Jordan curve, then by the 
Sch\"onflies theorem\index{Sch\"onflies theorem}  
there exists a homeomorphism $\varphi\: S^2 \ra \CDach$ such that $\varphi(J)=\partial \D$. In particular,
the set $S^2\setminus J$ has two connected components, both homeomorphic to $\D$.  Note that arcs and closed Jordan regions are cells of dimension $1$ and $2$, respectively. 

Let  $\DD$  be  a cell decomposition of $S^2$. Since the topological dimension of $S^2$ is equal to $2$,  no cell in $\DD$ can have  dimension $>2$. 
We call the $2$-dimensional cells in $\DD$ the  
{\em tiles},\index{tile} 
and  the $1$-dimensional cells in $\DD$ the 
{\em edges}\index{edge} 
of $\DD$. The 
{\em vertices}\index{vertex} 
of $\DD$ are the points $v\in S^2$ such that $\{v\}$ is a cell in $\DD$ of dimension $0$. So there is a somewhat subtle distinction between  vertices and cells of dimension $0$: a vertex is an element of $S^2$, while a cell of dimension $0$ is a subset of $S^2$ with one element.

 If $c$ is a cell in $\DD$, we denote by $\partial c$  the boundary and by 
$\inte(c)$ the interior of $c$ as introduced in the beginning of Section~\ref{s:celldecomp}. Note that  
for edges and $0$-cells $c$ this is different from the boundary and the interior of $c$ as a subset of the topological space $S^2$.

We now summarize some facts related to orientation. See Section~\ref{sec:orient} for a more detailed discussion.

We always assume that the sphere $S^2$ is {\em oriented}, i.e., one of the two generators 
of the singular homology group $H_2(S^2)\cong \Z$ (with coefficients in $\Z$) has been chosen as  the {\em fundamental class} of $S^2$.

The orientation on $S^2$ induces an orientation on every Jordan region $X\sub S^2$ which in turn 
induces an orientation on $\partial X$ and  on  every arc $\alpha \sub \partial X$. 
Here   an orientation of an arc is just a selection of one of the endpoints as the {\em initial point} and  the other endpoint as  the {\em terminal point}.
Let  $X\sub S^2$ be  a Jordan region in the oriented $2$-sphere $S^2$ equipped with the induced orientation.  If  $\alpha\sub \partial X$ is an arc with a given orientation, then we say that $X$ lies {\em to the left} or {\em to the right} of $\alpha$ depending 
on whether the orientation on $\alpha$ induced by the orientation of $X$ agrees with the given orientation on $\alpha$ or not. 
Similarly, we say that with a given orientation of $\partial X$ the Jordan region $X$ lies to the left or right of $\partial X$. 

To describe orientations, one can also use the notion of a
flag. By definition a {\em flag}\index{flag}    
in $S^2$ is a triple $(c_0,c_1, c_2)$, where $c_i$ is an $i$-dimensional cell for $i=0,1,2$,  $c_0\sub \partial c_1$, and $c_1\sub \partial c_2$.  So a flag in $S^2$ is a closed Jordan region  $c_2$ with an arc $c_1$ contained   in its boundary, where the point in $c_0$ is a distinguished  endpoint of $c_1$. We  orient  the arc $c_1$  so that the 
point in $c_0$ is the initial point in $c_1$.  The flag is called {\em positively-} or {\em negatively-oriented} (for the given orientation on $S^2$) depending on whether $c_2$ lies to the left or to the right of the oriented arc  $c_1$. 

A positively-oriented  flag determines the orientation on $S^2$ uniquely. The standard orientation on $\CDach$ is the one for which  the {\em standard flag} 
$(c_0, c_1, c_2)$ is positively-oriented, where 
$c_0=\{0\}$, $c_1=[0,1]\sub \R$, and 
$$c_2=\{z\in \C: 0\le \real(z)\le 1,\ 0\le 
\imag(z) \le \real(z)\}. $$    

Since edges and tiles 
in a cell decomposition $\DD$ of $S^2$ are arcs and closed Jordan regions, respectively, it makes sense to speak of oriented edges and tiles in 
$\DD$. A {\em flag in $\DD$} is a flag  $(c_0,c_1, c_2)$, where $c_0,c_1,c_2$ are cells in $\DD$. If $c_i$ are  $i$-dimensional cells in $\DD$ for $i=0,1,2$, then $(c_0,c_1,c_2)$ is  a flag in $\DD$ if and only if $c_0\sub c_1\sub c_2$.

After these preliminary remarks, we now turn to cell decompositions of a $2$-sphere $S^2$. They   have special  properties  summarized in the next lemma.  

\begin{lemma}\label{lem:specprop}
Let $\DD$ be a cell decomposition of $S^2$. Then it has the following properties:


\begin{enumerate}

\item
\label{item:prop_cell1}
  There are only finitely many cells in $\DD$.

\item
\label{item:prop_cell2}
  The tiles in $\DD$ cover $S^2$. 

\item
\label{item:prop_cell3}
  Let $X$ be a tile in  $\DD$.  Then  there exists a number $k\in \N$, $k\ge 2$, such that $X$ contains precisely $k$ edges $e_1, \dots, e_k$ and $k$ vertices
$v_1, \dots, v_k$ in $\DD$. Moreover,  these edges and vertices lie on the boundary  $\partial X$ of $X$,  and we have
$$\partial X=e_1\cup \dots \cup e_k. $$  
The indexing of these vertices  and edges can be chosen 
such that $v_j\in  \partial e_j\cap \partial e_{j+1}$
for $j=1, \dots, k$ (where $e_{k+1}\coloneqq e_1$).   

\item
\label{item:prop_cell4}
  Every edge $e\in \DD$ is contained in the boundary of  precisely two tiles $\DD$.  If $X$ and $Y$ are these tiles, then  $\inte(X)\cup \inte(e)\cup \inte(Y)$ is a simply connected region.  

\item
\label{item:prop_cell5}
  Let $v$ be a vertex of $\DD$. Then there exists a number $d\in \N$, $d\ge 2$,  such that  $v$ is contained in precisely $d$ tiles 
$X_1, \dots , X_{d}$, and $d$ edges $e_1, \dots, e_{d}$ in $\DD$.
We have $v\in \partial X_j$ and $v\in \partial e_j$ for each $j=1, \dots, d$.  Moreover, the indexing of these tiles and edges can be chosen 
such that $e_j\sub \partial X_j\cap \partial X_{j+1}$
for $j=1, \dots, d$ (where $X_{d+1}\coloneqq X_1$).  

\item
\label{item:prop_cell6}
  The $1$-skeleton of $\DD$ is connected and equal to the union of all edges in $\DD$. 

\end{enumerate}
\end{lemma}
 
Statement  \ref{item:prop_cell3} actually holds for all tiles  (i.e., $2$-dimensional cells)  in each   cell decomposition of a locally compact space.
If the boundary of a tile $X$ is subdivided into vertices and edges as
in \ref{item:prop_cell3}, we say that $X$ is a {\em (topological) $k$-gon}. 

If the edge $e$  and the tiles $X$ and $Y$ are as in \ref{item:prop_cell4}, then there exists a unique orientation of $e$ such that $X$ lies to the left and $Y$ to the right of $e$.

We say that the cells $\{v\}, e_1, \dots, e_d, X_1, \dots, X_d$  as in \ref{item:prop_cell5} form the {\em cycle}\index{cycle of vertex}  
of the vertex $v$ and call $d$ the {\em  length}  of the
cycle.\index{length!of cycle}  
We  refer to  $X_1, \dots, X_d$ as the tiles and to  $e_1, \dots, e_d$ as the   
 edges of the cycle (see Figure~\ref{fig:cycle} for an illustration). 

\ifthenelse{\boolean{nofigures}}{}{ 
  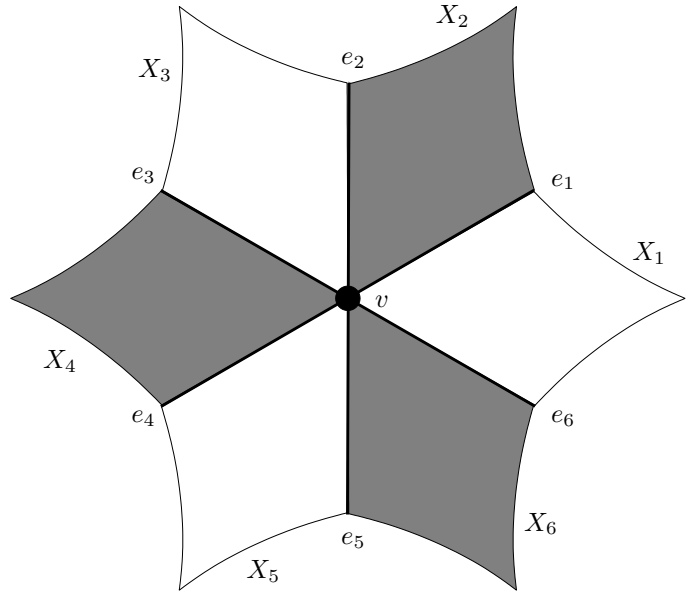
\begin{figure}
    \centering
    \begin{overpic}
      [width=9cm, 
      tics =20]{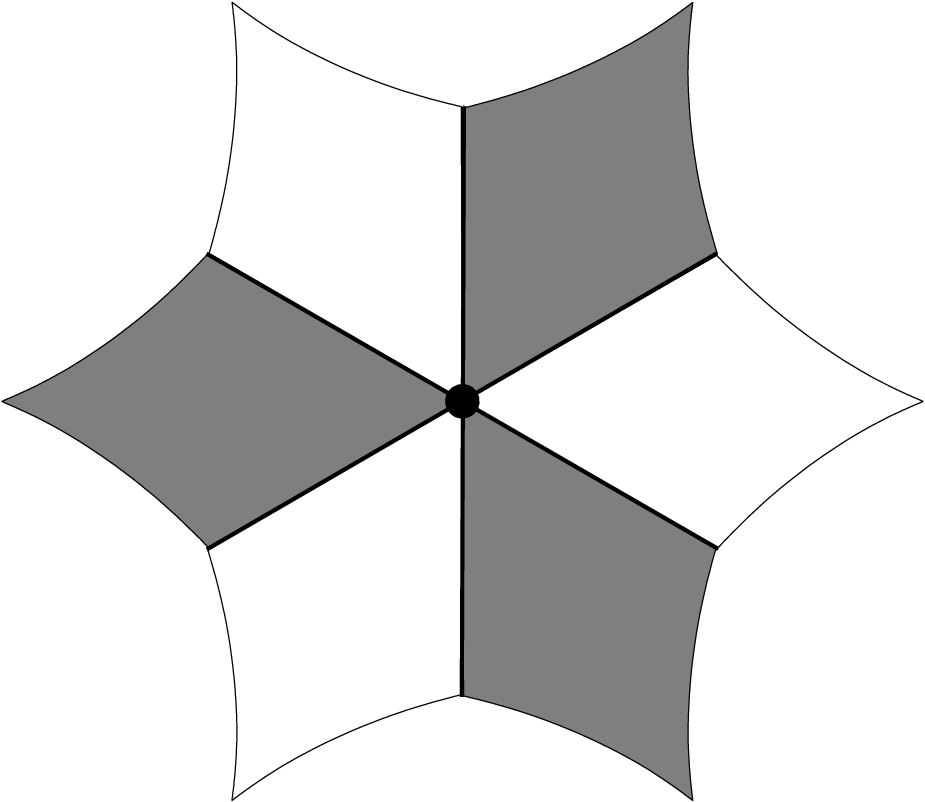}
      \put(92,49){$X_1$}
      \put(63,84){$X_2$}
      \put(19,76){$X_3$}
      \put(5,33){$X_4$}
      \put(35,2){$X_5$}
      \put(76,9){$X_6$}
      \put(80,60){$e_1$}
      \put(49,78){$e_2$}
      \put(18,61){$e_3$}
      \put(18,25){$e_4$}
      \put(49,7){$e_5$}
      \put(80,25){$e_6$}
      \put(54,42){$v$}
    \end{overpic}
    \caption{The cycle of a vertex $v$.}
    \label{fig:cycle}
  \end{figure}
}

\begin{proof} \ref{item:prop_cell1}  This follows from the compactness of $S^2$ and the fact that every point in $S^2$ has a neighborhood that meets only finitely many cells in $\DD$ (see Definition~\ref{def:celldecomp}~\ref{item:def_cell4}). 

\smallskip
\ref{item:prop_cell2} The set consisting of all vertices and the union of all edges 
has empty interior (in the topological sense) by \ref{item:prop_cell1} and Baire's theorem. Hence the union of all tiles is a dense set in $S^2$. Since this union is also closed by \ref{item:prop_cell1}, it is all of $S^2$. 

\smallskip
\ref{item:prop_cell3} Let $X$ be a tile in $\DD$. Then $\inte(X)$ does not meet any edge or vertex, and  $\partial X$ is a union of 
edges and  vertices. Since there are only finitely many vertices, $\partial X$ must contain an edge, and hence at least two vertices.  

Suppose   $v_1, \dots, v_k$, $k\ge 2$,  are  all the 
vertices on $\partial X$.  Since $\partial X$ is a Jordan curve,   we can choose the indexing of these vertices so that $\partial X$ is a union of arcs $\alpha_j$ with pairwise disjoint interior 
such that $\alpha_j$ has the endpoints  $v_j$ and $v_{j+1}$ for $j=1, \dots, k$, where $v_{k+1}=v_1$. Then for each $j=1, \dots, k$ the set $\inte(\alpha_j)$ is  connected and lies  in the $1$-skeleton of the cell decomposition $\DD$. It is disjoint from the $0$-skeleton and has  boundary contained in the $0$-skeleton. It follows from Lemma~\ref{lem:conncomp} and Lemma~\ref{lem:opencells} that there exists an edge
$e_j$ in $\DD$ with $\inte(e_j)=\inte(\alpha_j)$. Hence 
$\alpha_j=e_j$, and so $\alpha_j$  is an edge in $\DD$. 
It is clear that $\partial X$ does not contain other edges in $\DD$. The statement follows. 

\smallskip
\ref{item:prop_cell4} Let $e$ be an edge in $\DD$. Pick
$p\in\inte(e)$. By \ref{item:prop_cell2} the point $p$ is contained in some tile $X$ in $\DD$. 
By Lemma~\ref{lem:celldecompint}~\ref{item:cell_decomp2} we have $e\sub X$. 
On the other hand, $\inte(X)$ is disjoint from each edge and so $e\sub \partial X$.  It follows from the Sch\"onflies  theorem 
that the set $X$ does not contain a  neighborhood of $p$. 
Hence every neighborhood of $p$ must meet tiles distinct from $X$. Since there are only finitely many tiles, it follows that 
there exists a tile $Y$ distinct from $X$ with $p\in Y$.  By the same reasoning as before, we have  $e\sub \partial Y$. 
 
Let $q\in \inte(e)$ be arbitrary. Then there exists a small open and connected neighborhood $W$ of $q$ such that $W\setminus \inte(e)$ consists of two connected components $U$ and $V$.  If $W$ is small enough, then $U$ and $V$ 
do not  meet $\partial X$. Since $q\in\overline {\inte(X)}$,  one of the sets, say $U$, meets  $\inte(X)$, and so $U\sub \inte(X)$.    We can also assume that the set $W$ is small enough so that it does not meet $\partial Y$ either. By the same reasoning, $U$ or $V$ must be contained in $\inte(Y)$, and, since $\inte(X)\cap \inte(Y)=\emptyset$, we have $V\sub \inte(Y)$. 
Hence  
$M\coloneqq \inte(X)\cup \inte(e)\cup \inte(Y)$ contains the connected 
neighborhood $W\sub U\cup \inte(e)\cup V$  of $q$. Since $q\in \inte(e)$ was arbitrary, this implies that $M$ is open.  
 The sets $\inte(X)$,  $\inte(e)$, $\inte(Y)$ are connected, 
and their union $M$ contains a connected neighborhood of each point in $\inte(e)$. It follows that 
 $M$ is  connected. So $M$ is a region.
 
To see that $M$ is simply connected, first note that the Sch\"onflies theorem implies 
there exists a homotopy on $\inte(X) \cup \inte(e)$ that deforms this set into $\inte(e)$ and 
keeps the points in $\inte(e)$ fixed during the homotopy. If we combine this  homotopy  with 
a similar homotopy on  $\inte(Y) \cup \inte(e)$, then we see that $M$ is homotopic to 
$\inte(e)$, and hence also to a point. So $M$ is simply connected.

Suppose that  $Z$ is another tile in $\DD$  with $e\sub \partial Z$. 
Since  
$X\cup Y$ contains an open neighborhood for  $p$,  there exists a point $x\in \inte(Z)$ near $p$ with $x\in X\cup Y$, say $x\in X$.  Since the interior of a 
tile is disjoint from all other cells, we conclude $X=Z$.
This shows the  uniqueness of $X$ and $Y$. 

\smallskip
\ref{item:prop_cell5} Let $v$ be a vertex of $\DD$. If an edge $e$ in $\DD$ contains $v$, then $v$ is an endpoint of $e$ and we orient $e$ so that $v$ is the initial point of $e$. By \ref{item:prop_cell2} there exists a tile $X_1$ in $\DD$ with $v\in X_1$. Then $v\in \partial X_1$, 
and so  by \ref{item:prop_cell3} there exist two edges in $\partial X_1$ that contain $v$. For one of these oriented edges, which we denote  by  $e_1$, the tile $X_1$ will lie on the right  of $e_1$.  Then $v\in e_1\sub \partial X_1$ and  $X_1$ will lie on the left of the other oriented edge.  

By \ref{item:prop_cell4} there exists a unique tile $X_2\ne X_1$ 
with $e_1\sub \partial X_2$. Then $X_2$ will lie on the left of $e_1$.  By \ref{item:prop_cell3} there exists a unique  edge $e_2\sub  \partial X_2$ distinct from $e_1$ with $v\in e_2$. The tile $X_2$ will lie on the right of $e_2$. 
We can continue in this manner to obtain   tiles  
 $X_1, X_2, \dots$  and edges   $e_1, e_2, \dots$ that
  contain $v$ and satisfy $X_j\ne X_{j+1}$, $e_j\ne e_{j+1}$, and $e_j\sub \partial X_j\cap \partial X_{j+1}$ for all $j\in \N$.  Moreover, $X_j$ will lie on the right and $X_{j+1}$ on the left of the oriented edge $e_j$.  Since 
there are only finitely many tiles, there exists a smallest number $d\in \N$ such that the tiles $X_1, \dots, X_d$ are all distinct  and $X_{d+1}$ is equal to one of the tiles $X_1, \dots, X_d$. Since $X_1\ne X_2$, we have $d\ge 2$.

In addition, $X_{d+1}=X_1$.  To see this, we argue by contradiction and assume that $X_{d+1}$ is equal to one 
of the tiles $X_2, \dots, X_{d}$ 
say $X_{d+1}=X_j$. Note that $X_d\ne X_{d+1}$, so 
$2\le j\le d-1$. Then $e\coloneqq e_{d}$ is  an edge with $v\in e$ that is  contained in  $\partial X_d$ and in 
$\partial X_{d+1}=\partial X_{j}$. Hence $e=e_{j-1}$  or $e=e_{j}$. 
Since $X_{d+1}=X_j$ lies on the left of $e=e_d$, we must have 
$e=e_{j-1}$. Then $e$ is contained in the boundary of the three distinct tiles  $X_{j-1}, X_j, X_d$ which  is impossible by \ref{item:prop_cell4}. So indeed $X_{d+1}=X_1$. 
 
 By a similar reasoning we can show that the edges $e_1, \dots, e_d$ are all distinct. Indeed, suppose $e\coloneqq e_j=e_k$, where $1\le j<k\le d$. Then $k>j+1$ and  $e$ is contained in the boundary
 of the three distinct  tiles $X_{j}, X_{j+1}, X_k$ which is again absurd. 
 
 To show that there are no other edges and tiles containing 
 $v$ note that by \ref{item:prop_cell4} the set 
 $$U=\inte(X_1)\cup \inte(e_1)\cup \inte(X_2)\cup \dots \cup \inte(e_d)\cup \inte(X_{d+1})$$ 
 is open. Moreover, its boundary $\partial U$ consists 
 of the point $v$ and a closed set 
 $$A\sub  \bigcup_{j=1}^d \partial X_j $$
 disjoint from $\{v\}$. Hence $v$ is an isolated boundary point of $U$ which implies that $W=U\cup\{v\}$ is an open neighborhood of $v$. 
 
If $c$ is an arbitrary  cell in $\DD$  with $v\in c$ and $c\ne \{v\}$, then $v\in  \overline{\inte(c)}$. This  implies that $\inte(c)$ meets $U$. Since interiors of distinct cells in $\DD$ are disjoint, this is only possible if $c$ is equal to one of the edges $e_1, \dots, e_d$ or one of the tiles $X_1, \dots, X_d$.  The statement follows. 

\smallskip
\ref{item:prop_cell6} By \ref{item:prop_cell5} every vertex is contained in an edge. Hence the $1$-skeleton $E$ of $\DD$ is equal to the union of all edges in $\DD$. To show that $E$  is connected, let 
$x,y\in E$ be arbitrary. Since the tiles in $\DD$ cover $S^2$, there exist 
tiles $X$ and $Y$ with $x\in X$ and $y\in Y$. The interior of each tile is disjoint from the $1$-skeleton $E$, and so  $x\in \partial X$ and $y\in \partial Y$.  
Since $S^2$ is connected, there exist tiles $X_1, \dots, X_N$ in $\DD$  such that 
$X_1=X$, $X_N=Y$, and 
$X_i\cap X_{i+1}\ne \emptyset$ for $i=1, \dots, N-1$. 
The interior of a tile meets no other tile. Hence $\partial X_i\cap \partial X_{i+1}\ne \emptyset$ for  $i=1, \dots, N-1$. Since each set $\partial X_i$ is connected, it follows that 
$$K=\partial X_1\cup \dots \cup \partial X_N$$ is a connected subset of $E$
containing $x$ and $y$. 
This shows that  $E$  is connected. 
 \end{proof}
 
Let $d\in \N$, $d\ge 2$, and the tiles $X_j$ and edges $e_j$
for $j\in \N$  
be as defined in the proof of statement \ref{item:prop_cell5} of the previous lemma. Then we showed that $X_{d+1}=X_1$, but it is useful to point out that actually $X_j=X_{d+j}$  and $e_{j} =e_{d+j}$ for all $j\in \N$. 

Indeed we have seen that $X_{d+1}=X_1$. Moreover, 
$e_1, e_{d}, e_{d+1}$ are edges in $\DD$ that contain $v$ and are contained in the boundary of the tile $X_1=X_{d+1}$. Since there are only two such edges, $e_1\ne e_d$,  and $ e_{d}\ne e_{d+1}$,  we conclude that $e_{d+1}=e_1$.  Then $e_1=e_{d+1}$ is an edge contained in the boundary of the tiles $X_1, X_2, X_{d+2}$. Since there are precisely two tiles containing an edge in its boundary,  $X_1\ne X_2$, and $X_1=X_{d+1}\ne X_{d+2}$, it follows that $X_{d+2}=X_2$. 

If we continue in this manner, shifting all indices by $1$ in each step, we see that $e_{d+2}=e_2$, $X_{d+3}=X_3$, 
etc., as claimed.  

If we  orient 
$e_j$ so that $v$ is the initial point of $e_j$, then $X_{j+1}$ lies to  the left and $X_j$ lies to the right of $e_j$. So for each  $j\in \N$ 
the flag $(\{v\}, e_j, X_{j+1})$ in $\DD$ is  positively-oriented, and the flag 
$(\{v\}, e_j, X_{j})$ is negatively-oriented.  

 We will now discuss how cell decompositions   can be used to define
 homeomorphisms and isotopies of the underlying spaces.  We start with
 a general definition.  
 
 \begin{definition}[Isomorphisms of cell complexes]
   \label{def:compiso}
   \index{cell!complex!isomorphism}
   \index{isomorphism!of cell complexes}
    Let  $\DD$ and $\widetilde{\DD}$ be   cell complexes. A bijection 
 $\phi\: \DD\ra \widetilde{\DD}$ is called an {\em isomorphism (of
   cell complexes)} if the following conditions are satisfied: 
 \begin{enumerate}
 
 \item
   \label{item:compiso1}
   $\dim (\phi(\tau))=\dim (\tau)$ for all $\tau\in \DD$.

 \item
   \label{item:compiso2}
   If  $\sigma,\tau \in \DD$, then $\sigma\sub \tau$ if and only if 
 $\phi(\sigma)\sub \phi(\tau)$.
\end{enumerate}
\end{definition}
 
Let $h\colon \mathcal{X}\to \widetilde{\mathcal{X}}$ be a
homeomorphism between two locally compact Hausdorff spaces
$\mathcal{X}$ and $\widetilde{\mathcal{X}}$, and suppose $\DD$ is a
cell decomposition of $\mathcal{X}$. Then it is easy to see that
$\widetilde{\DD}\coloneqq  \{h(c) \subset \widetilde{\mathcal{X}} : c\in
\DD\}$ is a cell decomposition of $\widetilde{\mathcal{X}}$ and $\phi\colon
\DD\to \widetilde{\DD}$ given by $\phi(c) = h(c)$ for $c\in \DD$ is an
isomorphism. We will see that this procedure can be reversed and one
can construct a homeomorphism from a given cell complex
isomorphism. For simplicity we will restrict ourselves to the case of
$2$-spheres.  As a preparation for the proof of the corresponding
Lemma~\ref{lem:isocellhomeo}, we first record some facts about
homeomorphisms and isotopies on subsets of $2$-spheres.
 
If $\alpha$ is an arc, then every homeomorphism $\varphi\: \alpha \ra
\alpha$ that fixes the endpoints of $ \alpha$ is isotopic to the
identity rel.~$\partial \alpha$. Indeed, we may assume that $\alpha$
is equal to the unit interval $I=[0,1]$.  Then $\varphi(0)=0$,
$\varphi(1)=1$, and $\varphi$ is strictly increasing on $[0,1]$.
Define $H\: I\times I\ra I$ by
 $$H(s,t)=(1-t)\varphi(s)+ts$$ for $s,t\in I$. Then  $H_t(0)=0$,  $H_t(1)=1$, and   the map $H_t=H(\cdot,t)$ is strictly increasing on $I$ for each $t\in I$. It follows that $H$ is an isotopy rel.\ 
 $\partial I=\{0,1\}$. We have  
 $H_0=\varphi$ and $H_1=\id_I$, and so
 $\varphi$ and $\id_I$ are isotopic rel.~$\partial I$ by the isotopy $H$.
 
 Let $X\sub S^2$ be a closed Jordan region. If  $h\: \partial X\times I\ra \partial X$ is  an isotopy   with $h(\cdot, 0)=\id_{\partial X}$,  
 then there exists an isotopy $H\: X\times I\ra X$ such that $H(\cdot, 0)=\id_X$ and $H(p,t)=h(p,t)$ for all $p\in \partial X$ and $t\in I$. 
So an  isotopy $h$ on the boundary of $X$ with $h_0=\id_{\partial X}$ can be extended to an isotopy $H$ on $X$
with $H_0=\id_X$.
To see this, we may assume that $X=\overline \D$. Then 
$H$ is obtained from $h$ by radial extension; more precisely, we define $$H(re^{\iu s},t)=rh(e^{\iu s},t)$$
 for all $r\in [0,1]$ and $s\in [0,2\pi]$. Then $H$ is well-defined and it is easy to see that $H$ is an isotopy with the desired properties.  By using the Sch\"onflies theorem and  a similar radial extension one can also show that if $X$ and $X'$ are closed Jordan regions in $S^2$, then every homeomorphism $\varphi\: \partial X\ra \partial
 X'$ extends to a homeomorphism $\Phi\: X\ra X'$.  
 
If $\varphi \:X\ra X$ is a homeomorphism with 
$\varphi|\partial X=\id_{\partial X}$, then $\varphi$ is isotopic to $\id_{X}$ rel.~$\partial X$.  Indeed, again we may assume that $X=\overline \D$. Then we obtain the desired isotopy by the ``Alexander trick'': for $z\in \overline \D$ and $t\in I$ we define 
$ H(z,t)=t\varphi(z/t) $ if  $|z|< t$, and $H(z,t)=z$ if $|z|\ge t$. It is easy to see that $H$ is  an isotopy rel.\ $\partial \D$ with $H_0=\id_X$ and $H_1=\varphi$.      

If $\varphi, \widetilde \varphi  \:X\ra X$ are  two  homeomorphisms with 
$\varphi|\partial X=\widetilde \varphi |\partial X$, then we can apply the previous remark to $\psi= \widetilde \varphi^{-1} \circ \varphi$  and conclude that 
$\varphi$  and $\widetilde \varphi$ are isotopic rel.\ $\partial X$.

We are now ready to state and prove a fact that allows us to construct homeo\-morphisms from 
cell complex isomorphisms.

 \begin{lemma} 
   \label{lem:isocellhomeo}
   \index{cell!complex!isomorphism} 
   \index{isomorphism!of cell complexes} 
   Let $\DD$ and $\widetilde \DD$ be isomorphic cell
   decompositions of $2$-spheres $S^2$ and $\widetilde S^2$,
   respectively, and let $\phi \: \DD \ra \widetilde \DD$ be an
   isomorphism. Then the following statements are true:
 
 \begin{enumerate}
 
 \item
   \label{item:isocellhomeo1}
  If $h\: S^2 \ra \widetilde S^2$ is a map such that $h|\tau$ is a homeomorphism of $\tau$ onto 
  $\phi(\tau)$ for each $\tau\in \DD$, then $h$ is a homeomorphism of $S^2$ onto $\widetilde S^2$. 
  
 \item
   \label{item:isocellhomeo2}
There exists a homeomorphism $h\: S^2 \ra \widetilde S^2$ such that $h(\tau)=\phi(\tau)$ for all 
$\tau \in \DD$. 

  \item
   \label{item:isocellhomeo3} Let ${\bf V}$ be the set of vertices of $\DD$. 
If $h_0,h_1\: S^2 \ra \widetilde S^2$ are two homeomorphisms with 
 $h_0(\tau)=\phi(\tau)=h_1(\tau)$ for all 
$\tau \in \DD$, then $h_0$ and $h_1$ are isotopic rel.\ ${\bf V}$.

  \end{enumerate}
\end{lemma}

If  $h$ is as in \ref{item:isocellhomeo2}, then we say that $h$
{\em realizes}\index{map!realizing!isomorphism}\index{realizing!isomorphism} 
the 
cell complex isomorphism $\phi$. So an isomorphism between cell decompositions   of $2$-spheres can always be realized by a homeomorphism $h$ and  by  \ref{item:isocellhomeo3} this homeomorphism is unique up to isotopy.  A similar fact is actually true in greater generality, but 
Lemma~\ref{lem:isocellhomeo} will be enough for our purposes.
  
\begin{proof} In the following we write $\widetilde \tau\coloneqq \phi(\tau)$ for $\tau \in \DD$.

\smallskip 
 \ref{item:isocellhomeo1} Let $h\: S^2 \ra \widetilde S^2$ be a
 map such that $h|\tau$ is a homeomorphism of $ \tau$ onto
 $\widetilde \tau$ for each $\tau\in \DD$. Then $h$ is
 continuous, because the restriction $h|\tau$ is continuous for
 each $\tau\in \DD$ and the cells $\tau \in \DD$ form a  finite
 cover of $S^2$ by closed sets. 
The image cells $\widetilde \tau=h(\tau)$ 
form the cell decomposition $\widetilde \DD$ of $\widetilde S^2$ and hence cover $\widetilde S^2$. So $h$ is also surjective. 
 
 In order to conclude that $h\: S^2 \ra \widetilde S^2$ is a homeomorphism, it suffices to show that $h$ is injective. To see this, let $x_1,x_2\in S^2$ and assume that $y\coloneqq h(x_1)=h(x_2)$. 
 Then there exist unique cells $\tau_1,\tau_2\in \DD$ such that
 $x_1\in \inte(\tau_1)$ and $x_2\in \inte(\tau_2)$. 
Since $h|\tau_i$ is a homeomorphism 
of $\tau_i$ onto $\widetilde \tau_i$ for $i=1,2$, we have $y\in \inte(\widetilde \tau_1)\cap \inte(\widetilde \tau_2)$. This implies that
 $\widetilde \tau_1=\widetilde \tau_2$. Since the map $\tau\in \DD\mapsto \widetilde \tau\in \widetilde \DD$ is an isomorphism, it follows  that  $\tau_1=\tau_2$. So $x_1$ and $x_2$ are contained in the same cell $\tau\coloneqq \tau_1=\tau_2\in \DD$. Since $h|\tau$ is a homeomorphism onto $\widetilde \tau$ and hence injective, we conclude that $x_1=x_2$ as desired.

\smallskip 
 \ref{item:isocellhomeo2} By \ref{item:isocellhomeo1} it suffices to find a map $h\: S^2 \ra \widetilde S^2$ such that $h|\tau$ is a homeomorphism of $\tau$ onto $\widetilde \tau$ for each 
 $\tau\in \DD$.  The existence of $h$ follows  from the well-known procedure of successive extensions to  the skeleta of the 
cell decomposition $\DD$. 

Indeed, if $v$ is vertex in $\DD$, then there exists a unique vertex $\widetilde v$ in $\widetilde \DD$ such that $\phi (\{v\})=\{ \widetilde v\}$. We define $h(v)=\widetilde v$. Then $h$ is a bijection of the set of vertices of $\DD$ onto the set of vertices in $\widetilde \DD$. To extend 
$h$ from the $0$-skeleton of $\DD$ to  
 the $1$-skeleton, let $e$ be an arbitrary edge in $\DD$ and  $u$
 and $v$ be  the vertices in $\DD$   that are the endpoints of
 $e$. Then $\widetilde u$ and $\widetilde v$ are the endpoints of
 $\widetilde e$.  So we can extend $h$ to $e$ by choosing a
 homeomorphism of $e$ onto $\widetilde e$ that agrees with $h$ on
 the endpoints of $e$. In this way 
we can continuously extend $h$ 
to the $1$-skeleton of $\DD$ 
so that $h|\tau$ 
is a homeomorphism of $\tau$ onto 
 $\widetilde \tau$, whenever $\tau$ is a cell in $\DD$ with $\dim(\tau)\le 1$. An argument as 
 in the proof of  \ref{item:isocellhomeo1} shows that $h$ is a homeomorphism of the $1$-skeleton 
 of $\DD$ onto the $1$-skeleton of $\widetilde \DD$.
  
If $X$ is an arbitrary tile in $\DD$, then $\partial X$ is a subset of the $1$-skeleton of $\DD$  and hence $h$ is already defined on $\partial X$. Then   
$h|\partial X$ is an injective and  continuous  mapping of $\partial X$ into the boundary $\partial \widetilde X$ of the tile $\widetilde X\in \widetilde \DD$.  Since  an  injective and continuous map of a   Jordan curve  into another Jordan curve is necessarily  surjective,
$h|\partial X$ is a homeomorphism of $\partial X$ onto $\partial \widetilde X$. Hence $h$ can be extended to a homeomorphism of $X$ onto $\widetilde X$. These extensions on different tiles
  paste together to a  map $h\:S^2\ra S^2$ with the desired property that   $h|\tau$ is a homeomorphism of $\tau$ onto $\widetilde \tau$ for each 
 $\tau\in \DD$. By  \ref{item:isocellhomeo1} the map $h$ is a homeomorphism of $S^2$ onto $\widetilde S^2$ with $h(\tau)=\widetilde \tau=\phi(\tau)$ for $\tau\in \DD$.
 
 \smallskip 
 \ref{item:isocellhomeo3} 
 Suppose $h_0,h_1\:S^2\ra \widetilde S^2$ are as in the statement. 
Then $\varphi\coloneqq h_1^{-1}\circ h_0$ is a homeomorphism on $S^2$ that maps each cell  $\tau\in \DD$ onto itself. In particular,   $\varphi$ is the identity on the set ${\bf V}$ of vertices of $\DD$. 

By  successive extensions to the $1$- and the $2$-skeleton  of $\DD$ we will show that $\varphi$ is actually isotopic to $\id_{S^2}$ rel.\ ${\bf V}$.  For this   we denote  the set of edges of $\DD$ by  $\E$ and by 
$E=\bigcup\{e: e\in \E\}$ the $1$-skeleton of $\DD$. Let $e\in \E$ be   an arbitrary edge in $\DD$. Since $\varphi(e)=e$ and $\varphi$ is the identity on ${\bf V}$, the map $\varphi|e$ is isotopic to $\id_e$ rel.~$\partial e$.  These isotopies on edges paste together to an isotopy of $\varphi|E$ to $\id_{E}$ rel.~${\bf V}$. If $X$ is a tile in $\DD$, then this isotopy is defined on $\partial X\sub E$, and we can extend it to an isotopy 
of a homeomorphism on $X$ that agrees with $\varphi|\partial X$ on $\partial X$    to $\id_X$. These extensions on tiles $X$  paste together 
to an isotopy $\Phi\: S^2 \times [0,1]\ra S^2$ rel.~${\bf V}$ such that 
$\psi |E=\varphi|E$, where $\psi\coloneqq \Phi(\cdot, 0),$ and $\Phi(\cdot, 1)=\id_{S^2}$.

For each tile $X\in \DD$ the maps 
 $\varphi|X$  and $\psi |X$ are homeomorphisms of $X$ onto itself that agree on $\partial X\sub E$. As we have seen in the discussion before the proof of the lemma, this implies that $\varphi|X$  and $\psi|X$ are isotopic rel.\ $\partial X$.  Again by pasting these isotopies on tiles together, we can find an isotopy $\Psi \: S^2 \times [0,1]\ra S^2$ rel.~$E$ with 
 $\Psi(\cdot, 0)=\varphi$ and $\Psi(\cdot, 1)=\psi$. 
 The concatenation of the isotopies $\Psi$ and $\Phi$   gives an isotopy 
 rel.\ $ {\bf V}$ between $\varphi=h_1^{-1}\circ h_0$ and $\id_{S^2}$.
 If we postcompose this isotopy with $h_1$, we get an isotopy between $h_0$ and $h_1$ rel.\ ${\bf V}$ as desired. 
 \end{proof}

 \section{Cell decompositions  induced by Thurston maps}\index{cell!decomposition!induced by Thurston map}
\label{sec:tiles}
 

Let $f\: S^2\ra S^2$ be a Thurston map, and $\CC\sub S^2$ be a
Jordan curve such that $\post(f)\sub \CC$. In this section we
will discuss how the pair $(f,\CC)$ induces natural cell
decompositions of $S^2$.
 
By the Sch\"onflies theorem there are two closed Jordan regions
$\XOb, \XOw\subset S^2$\index{X0@$\XOw,\XOb$} whose boundary is
$\CC$. Our notation for these regions is suggested by the fact
that we often think of $\XOb$ as being assigned or carrying the
color ``black'', represented by the symbol ${\tt b}$, and $\XOw$
as being colored ``white'' represented by ${\tt w}$.  We will
discuss this more precisely later in this section (see Lemma~\ref{lem:colortiles}).

The sets $ \XOb$ and $\XOw$ are topological cells of dimension
$2$. We call them \defn{tiles of level $0$} or $0$-{\em
  tiles}. The postcritical points of $f$ are on the boundary of
$\XOw$ and $\XOb$. We consider them as \defn{vertices} of $\XOw$
and $\XOb$, and the closed arcs of $\CC$ between vertices as the
\defn{edges} of the $0$-tiles. In this way, we think of $\XOw$
and $\XOb$ as topological $m$-gons where $m=\#\post(f)\ge 2$ (see
Corollary~\ref{cor:post012}). 
To emphasize that these edges and
vertices belong to $0$-tiles, we call them $0$-{\em edges} and
$0$-{\em vertices}.  

A $0$-{\em cell} is a $0$-tile, a $0$-edge,
or a set consisting of a $0$-vertex. Obviously, the $0$-cells
form a cell decomposition of $S^2$ that we denote by
$\DD^0=\DD^0(f,\CC)$. Roughly speaking, we can now obtain  cell decompositions
$\DD^n(f,\CC)$ of $S^2$  for each $n\in \N_0$ by taking preimages of
$\DD^0(f,\CC)$ under $f^n$.   This is based on 
the following lemma. 

\begin{lemma} 
  \label{lem:pullback} 
  Let $f\: S^2\ra S^2$ be a branched covering map,  and $\DD$ be a
  cell decomposition of $S^2$ such that every point in
  $f(\crit(f))$ is a vertex in $\DD$. Then there exists a unique
  cell decomposition $\DD'$ of $S^2$ such that $f$ is cellular
  for $(\DD', \DD)$.
\end{lemma}

In general, $\DD'$ will not be  a refinement of
$\DD$. As we will see in the proof, $\DD'$ consists precisely  of all cells
  $c\sub S^2$ such that $f(c)$ is a cell in $\DD$ and $f|c$ is a
  homeomorphism of $c$ onto $f(c)$. In particular, if ${\bf V}$ and ${\bf V}'$ denote the 
set of    vertices of $\DD$ and $\DD'$, respectively, then  ${\bf V'}=f^{-1}({\bf V})$. 


If, in the setting of Lem\-ma~\ref{lem:pullback}, we make the stronger assumption $\post(f)\sub {\bf V}$, then $f^n(\crit(f^n))\sub \post(f) \sub {\bf V}$ for all $n\in \N$ and we can apply the lemma to all iterates of $f$.

Before we prove Lemma~\ref{lem:pullback}, we record some immediate  
 consequences.

\begin{cor}
  \label{cor:pullback_D0}
  Let $f\colon S^2\to S^2$ be a Thurston map, $\CC\subset S^2$ be
  a Jordan curve with $\post(f) \subset \CC$, and $\DD^0(f,\CC)$
  be defined as above. Then there exists a unique sequence of cell decompositions
  $\DD^n=\DD^n(f,\CC)$, $n\in \N_0$,  such that $f$ is cellular
  for $(\DD^{n+1},\DD^n)$ for each $n\in \N_0$.
\end{cor}

\begin{proof}
  Since the  points in $\post(f)\supset f(\crit(f))$ form the vertices in  $\DD^0=\DD^0(f,\CC)$, we can
  apply Lem\-ma~\ref{lem:pullback} to obtain a  cell
  decomposition $\DD^1=\DD^1(f,\CC)$ such that $f$ is cellular
 for $(\DD^1, \DD^0)$. By the remark following Lemma~\ref{lem:pullback}  
  a point is a vertex in  $\DD^1$ precisely if its image under $f$ is a vertex of $\DD^0$. 
  So the  set of vertices of $\DD^1$ is given by   
  $f^{-1}(\post(f))\supset \post(f)\supset f(\crit(f))$.
   Hence we can apply
  Lemma~\ref{lem:pullback} again and obtain a cell decomposition
  $\DD^2=\DD^2(f,\CC)$ such that $f$ is cellular for
  $(\DD^2, \DD^1)$.  Continuing in this manner, we obtain cell
  decompositions $\DD^n=\DD^n(f,\CC)$ of $S^2$ for $n\in \N_0$
  such that $f$ is cellular for $(\DD^{n+1}, \DD^n)$ for all
  $n\in \N_0$.  
 
 The  last property uniquely   determines the cell decompositions $\DD^n=\DD^n(f,\CC)$ for all $n\in \N_0$ as follows from the uniqueness statement in Lemma~\ref{lem:pullback}. 
\end{proof}

The cell decompositions $\DD^n(f,\CC)$  will be used throughout this
work. 

\begin{definition}[Cell decompositions for $f$ and $\CC$]
  \index{d00@$\DD^n(f,\CC)$|textbf}
  \label{def:DDn}
  Given a Thurston map $f\colon S^2\to S^2$ and a Jordan curve
  $\CC\subset S^2$ with $\post(f) \subset \CC$, the cell
  decompositions $\DD^n(f,\CC)$ for $n\in \N_0$ are the ones provided by  Corollary~\ref{cor:pullback_D0}. 
\end{definition}

We call the elements in $\DD^n(f,\CC)$ the 
{\em $n$-cells}\index{n9@$n$-!cell}\index{for fC@for $(f,\CC)$}\index{n9@$n$-!cell!for $(f,\CC)$}
for $(f,\CC)$, or simply
$n$-cells if $f$ and $\CC$ are understood. We call $n$ the 
{\em level}\index{level of cell} 
of an $n$-cell. When we speak of
$n$-cells, then $n$ always refers to this level and not to the
dimension of the cell.  An $n$-cell of dimension $2$ is called an
{\em $n$-tile}\index{n9@$n$-!tile}, 
and an $n$-cell of dimension
$1$ an 
{\em $n$-edge}\index{n9@$n$-!edge}. 
An 
$n$-{\em vertex}\index{n9@$n$-!vertex} 
is a point $p\in S^2$ such that
$\{p\}$ is an $n$-cell of dimension $0$.  We denote the set of
all $n$-tiles, $n$-edges, and $n$-vertices for $(f,\CC)$ by
$\X^n(f,\CC)$,\index{Xaa@$\X^n$} $\E^n(f,\CC)$,\index{Eaa@$\E^n$}
and ${\bf V}^n(f,\CC)$,\index{Vaa@${\bf V}^n$} respectively. If
$f$ and $\CC$ are understood, we simply write $\X^n$ for
$\X^n(f,\CC)$, etc.

In Proposition~\ref{prop:celldecomp} we will record some properties 
of the cell decompositions $\DD^n(f,\CC)$ and also give a more explicit description of 
their cells. We first turn to the proof of Lemma~\ref{lem:pullback}. 
 We require a lemma.

\begin{lemma} \label{lem:continv} 
Let $X$ and $Y$ be metric spaces, and $f\: X\ra Y$ be a continuous map. Suppose $X$ is compact and $A\sub Y$ is closed. Then for each 
$\eps>0$ there exists $\delta>0$ such that 
$$ f^{-1}(\mathcal{N}_\delta(A))\sub \mathcal{N}_\eps(f^{-1}(A)).$$
\end{lemma}

Here  $\mathcal{N}_r(M)$ for $r>0$ denotes the open $r$-neighborhood of a set $M$ in a metric space. 

\begin{proof} We argue by contradiction and assume that for some $\eps>0$ the statement is not true. Then for each $n\in \N$ there exists
a point $x_n\in  f^{-1}(\mathcal{N}_{1/n}(A))$ with $x_n\not \in
 \mathcal{N}_\eps(f^{-1}(A))$. Since $X$ is compact, by passing to a subsequence if necessary, we may assume that $\{x_n\}$ converges, say $x_n\to x\in X$ as $n\to \infty. $
 Then $$\dist(x, f^{-1}(A))=\lim_{n\to \infty} \dist(x_n, f^{-1}(A))\ge \eps>0.$$ 
 On the other hand, $f(x_n)\in \mathcal{N}_{1/n}(A)$ for $n\in \N$ and $f(x_n)\to f(x)$ as $n\to \infty$.
This implies  that $f(x)\in \overline A=A$, and so $x\in f^{-1}(A)$.  This is a contradiction. 
  \end{proof}

\begin{proof}[Proof of Lemma~\ref{lem:pullback}]
  To show existence, we define $\DD'$ to be the set of all cells
  $c\sub S^2$ such that $f(c)$ is a cell in $\DD$ and $f|c$ is a
  homeomorphism of $c$ onto $f(c)$. It is clear that $\DD'$ does
  not contain cells of dimension $>2$.  As usual, we call the
  cells $c$ in $\DD'$ edges or tiles depending on whether $c$ has
  dimension $1$ or $2$, respectively. The vertices $p$ of $\DD'$
  are the points in $S^2$ such that $\{p\}$ is a cell in $\DD'$
  of dimension $0$.  It is clear that the set of vertices of
  $\DD'$ is equal to $f^{-1}({\bf V})$, where ${\bf V}$ is the
  set of vertices of $\DD$.
 
 In order to show that $\DD'$ is a cell decomposition of $S^2$, we first establish two claims.
 
 \smallskip 
 {\em Claim 1.} If $p\in S^2$ and    $q=f(p)\in \inte(X)$ for some tile $X\in \DD$, then there exists a unique tile $X'\in \DD'$ with $p\in X'$.  
\smallskip 
     
  In this case, let    $U=\inte(X)$. Then $U$  is an open and 
  simply connected set in the complement of ${\bf V}\supset f(\crit(f))$. Hence there exists a unique continuous map 
     $g\:U\ra U'\coloneqq g(U)\sub S^2$ (a ``branch of the inverse of $f^{-1}$'') with
  $f\circ g=\text{id}_U$ and $g(q)=p$. The map $g$ is a homeomorphism onto its image $U'$. Hence   
  $U'\sub S^2$ is  open and simply connected.  

  We equip $S^2$ with some base metric inducing the given topology. In the following, metric terms will refer to this metric. Then  it follows from Lemma~\ref{lem:continv} that $f$ has the following property:
  for all $w\in S^2$ and all $\eps>0$, there exists $\de>0$ such that
  \begin{equation} \label{prope1}
    f^{-1}(B(w,\delta)) \sub {\mathcal N}_\eps(f^{-1}(w)).
  \end{equation}
 
  We want to prove that $g$ has a continuous extension
  to $\overline U=X$. For this  it suffices to show that $\{g(w_i)\}$
  converges whenever $\{w_i\}$ 
  is a sequence in $U$ converging to a point $w\in \partial U$.  Since
  $g$ is a right inverse of $f$, it follows that the limit points of
  $\{g(w_n)\}$ are contained in $f^{-1} (w)$. Since $f$ is
  finite-to-one, 
  the point $w$ has finitely many preimages $z_1, \dots, z_m$ under
  $f$. 

  We can choose $\eps>0$ so small that the sets
  $B(z_i, \eps)$, $i=1, \dots, m$,  are pairwise disjoint.
  By \eqref{prope1} we can find $\de>0$ such that
  \begin{equation}\label{prope2}
   f^{-1}(B(w, \de))\sub \bigcup_{i=1}^m B(z_i, \eps).
  \end{equation}
  The set $\overline U=X$ is a closed Jordan region, and hence locally
  connected. So there exists an open connected set $V\sub U$ such
  that $\overline V$ is a neighborhood of $w$ in $\overline U$ and  
  $\overline V\sub B(w,\de)$. 
  Then $g(V)$ is a connected subset of $f^{-1}(B(w,\de))$. Since the union on the right hand side of 
  \eqref{prope2} is disjoint, the set 
  $g(V)$  must be contained in one of the sets of this union, say
  $g(V)\sub B(z_k, \eps)$. 
  Now $w_i \in V$ for sufficiently large $i$, and so
  all limit points of $\{g(w_i)\}$ are contained in
  $\overline {g(V)}\sub \overline B(z_k, \eps)$.
  On the other hand, the only possible limit points of $\{g(w_i)\}$ are
  $z_1, \dots, z_m$, and $z_k$ is the only one contained in 
  $\overline B(z_k, \eps)$. This implies 
  $g(w_i)\to z_k$ as $i\to \infty$. So  $g$ has indeed a continuous extension to
  $\overline U$, which is again denoted by $g$. 
  
   It is clear  that
  \begin{equation}\label{prope3}
    f \circ g=\text{id}_{\overline U}.
  \end{equation}
  This implies that $g$ is a homeomorphism of $\overline U=X$
  onto its image $X'\coloneqq g(\overline U)=\overline {g(U)}$.
  Then $X'$ is a closed Jordan region, and 
  by \eqref{prope3} the map $f|X'$ is a homeomorphism of $X'$ onto
  $\overline U=X$. Hence $X'$ is a tile in $\DD'$ with $p\in g(U)\sub X'$.  
  
  So a tile   $X'\in \DD'$ containing $p$ exists. We want to show that it is the only tile in $\DD'$ containing $p$. Indeed, suppose $Y'\in \DD'$  is another tile with $p\in Y'$. Then $f(Y')$ is a tile in $\DD$ containing the point $q=f(p)\in \inte(X)$. 
  Hence $f(Y')=X$, and so $f|Y'$ is a homeomorphism of $Y'$ onto $X$.  Let $h=(f|Y')^{-1}$.
  Then $g$ and $h$ are both inverse branches of $f$ 
  defined on the simply connected region $U$ with $g(q)=p=h(q)$. Hence $h$ and $g$ agree on $U$ (see Lemma~\ref{lem:liftsofcov}~\ref{item:liuniq}), and so by continuity also on $\overline U$. We conclude that  $X'=g(X)=
  h(X)=Y'$ as desired, and so Claim~1 follows. 
  
  \smallskip
  {\em Claim 2.}  If $p\in S^2$ and    $q=f(p)\in \inte(e)$ for some edge  $e\in \DD$, then there exists a unique edge  $e'\in \DD'$, and precisely two distinct tiles $X'$ and $Y'$ in $\DD'$ that contain $p$.  Moreover, $e'\sub \partial X'\cup \partial Y'$.  
\smallskip

By Lemma~\ref{lem:specprop}~\ref{item:prop_cell4} we know that there are precisely two distinct tiles $X,Y\in \DD$ that contain $e$ in their boundary, and  
 that $U=\inte(X)\cup \inte(e)\cup \inte(Y)$ 
 is   an open and simply connected region in the complement of the set ${\bf V}\supset f(\crit(f))$. 
 Hence there exists a unique continuous map $g\:U\ra S^2$ with $g(q)=p$ and $f\circ g=\id_{U}$.  Then $g$ is a homeomorphism of $U$ onto the open set $g(U)\sub S^2$. 
 As before one can show that the maps $g_1\coloneqq g|\inte(X)$ and $g_2\coloneqq g|\inte(Y)$ have continuous extensions to $X$ and $Y$, respectively.
  We use   the same notation   $g_1$ and $g_2$ for these extensions.  It is clear that $g_1|e=g_2|e$.
  Moreover, $g_1$ is a homeomorphism of $X$ onto a closed Jordan region $X'=g_1(X)$ with inverse map 
  $f|X'$. In particular, $X'$ is a tile in $\DD'$.  Similarly, $Y'= 
  g_2(Y)$ is a tile in $\DD'$.  
  The tiles $X'$ and $Y'$ are distinct, because $f$ maps them to different tiles in $\DD$. Moreover, $e'\coloneqq g_1(e)=g_2(e)$ is an edge in $\DD'$ with $p\in e'\sub \partial X'\cap \partial Y'$.

It remains to prove the uniqueness part. If $\widetilde e$ is another edge in $\DD'$ with $p\in \widetilde e$, then $f(\widetilde e)$ is an edge in $\DD$ and  $f|\widetilde e$  is a homeomorphism 
of $\widetilde e$ onto $f(\widetilde e)$. Hence $q=f(p)\in \inte(e)\cap f(\widetilde e)$ which implies that $f(\widetilde e)=e$. So $f|\widetilde e$ is actually a homeomorphism of $\widetilde e$ onto $e$. 
  Then $(f|\inte(e'))^{-1}$ and 
$(f|\inte(\widetilde e))^{-1}$ are right inverses of $f$ defined on the open arc $\inte(e)$ that both map $q$ to $p$. Hence these right inverses must agree on $\inte(e)$. By continuity this implies  $(f|e')^{-1}=(f|\widetilde e)^{-1}$ on $e$, and so $e'=(f|e')^{-1}(e)=(f|\widetilde e)^{-1}(e)=\widetilde e$.

If $Z'$ is another tile in $\DD'$ with $p\in Z'$, then $f$ maps $\partial Z'$ homeomorphically to the boundary $\partial f(Z')$ of the tile $f(Z')\in \DD$. Moreover,  $p\in \partial Z'$; for otherwise $f(p)$ would lie in the set $\inte(f(Z'))$ which is disjoint from 
$e$. 
It follows that  there is an edge in $\DD'$ that contains $p$ and is contained in the boundary of $\partial Z'$.  Since this edge 
in $\DD'$ is unique, as we have just seen, we know that $e'\sub \partial Z'$. 
Note that $p\in g(U) \sub X'\cup Y'$.  So 
$X'\cup Y'$ is a neighborhood of $p$, because $g(U)$ is open. Since $p \in e'\sub \partial Z'$, there exists a point $x\in \inte(Z')$ near $p$ with $x\in X'\cup Y'$, say $x\in X'$. 
Then $f(x)$ is contained in the interior of the  tile $f(Z')\in \DD$.   Since $x\in X'\cap Z'$ and $X'$ and $Z'$ are both tiles in $\DD'$,   we conclude $X'=Z'$ by Claim~1. This completes  the proof of Claim 2.  

\smallskip 
Now that we have established Claims~1~and~2, we can show that 
$\DD$ is a cell decomposition of $S^2$ by verifying conditions  \ref{item:def_cell1}--\ref{item:def_cell4} of  Definition~\ref{def:celldecomp}. 

\smallskip 
{\em Condition} \ref{item:def_cell1}: If $p\in S^2$ is arbitrary, then $f(p)$ is a vertex of $\DD$ or $f(p)$ lies in the interior of an edge or in the interior of a tile in $\DD$. In the first case $p$ is a vertex of $\DD'$, and in the other two cases $p$ lies in cells in $\DD'$ by Claim 1 and Claim 2. 
It follows that the cells  in $\DD'$ cover $S^2$.

\smallskip 
{\em Condition} \ref{item:def_cell2}: Let $\sigma, \tau $ be cells in $\DD'$ with $\inte(\sigma)\cap \inte(\tau)\ne \emptyset$. Then $f(\sigma)$ and $f(\tau)$ are cells in $\DD$ with $\inte(f(\sigma))\cap \inte(f(\tau))\ne \emptyset$. Hence $\lambda\coloneqq f(\sigma)=f(\tau)$ and $\lambda\in \DD$.
In particular, $\sigma$ and $\tau$ have the same dimension.

If $\sigma$ and $\tau$ are both tiles, then $\sigma=\tau$ by Claim 1, because every  point in  $\inte(\sigma)\cap \inte(\tau)\ne \emptyset $ has an image under $f$ in $\inte(\lambda)$. Similarly, if $\sigma$ and $\tau$ are edges, then $\sigma=\tau$ by Claim 2.

If $\sigma$ and $\tau$ consist of vertices in $\DD'$, then   the relation 
$\inte(\sigma)\cap \inte(\tau)\ne \emptyset$ trivially implies 
$\sigma=\tau$. 

\smallskip 
{\em Condition} \ref{item:def_cell3}: Let  $\tau'\in \DD'$ be arbitrary. Then $f|\tau'$ is a homeomorphism of $\tau'$ onto the cell $\tau=f(\tau')\in \DD$. Note that  $(f|\tau')^{-1}(\sigma)\in \DD'$ whenever $\sigma\in \DD$ and $\sigma\sub \tau$. Since 
$\partial \tau'= (f|\tau')^{-1}(\partial \tau)$ and $\partial \tau$ is a union of cells in $\DD$, it follows that $\partial \tau'$ is a union of cells in $\DD'$. 

\smallskip 
{\em Condition} \ref{item:def_cell4}: To establish the final property of a cell decomposition for $\DD'$, we will show that $\DD'$ consists of only finitely many 
cells. Indeed, let $N_i\in \N$ be the number of cells of dimension $i$ in $\DD$ for $i=0,1,2$. 
Since the vertices in $\DD'$ are the preimages of the vertices of $\DD$, we have at most $\deg(f) N_0$ vertices in $\DD'$. 
 
 Pick one  point in the interior of each edge in $\DD$. The set $M$  of  these points consists of $N_1$ elements. If $q\in M$, then $q\notin{\bf V}\supset f(\crit(f))$, and so $q$ is not  a critical value of $f$. Hence $\#f^{-1}(M)=N_1\deg(f)$. It follows from Claim 2 that each element of $f^{-1}(M)$ is contained in a unique 
 edge in $\DD'$,  and it follows from the definition of $\DD'$ that each edge in $\DD'$ 
contains a unique point in $f^{-1}(M)$. Hence the number of edges in $\DD'$ is equal to $\#f^{-1}(M)=N_1\deg(f)$. 

Similarly, pick a point in the interior of each tile in $\DD$ and let  $M$ be the set of these points. Then $\#f^{-1}(M)=N_2\deg(f)$ and by a similar  reasoning as above based on Claim 1, we see 
that the number of tiles in $\DD'$ is equal to $\#f^{-1}(M)=N_2\deg(f)$.

\smallskip 
We have shown that $\DD'$ is a cell decomposition of $S^2$. 
It follows immediately from the definition of $\DD'$ that $f$ is cellular for $(\DD', \DD)$. 

\smallskip 
To show uniqueness of $\DD'$,  suppose that $\widetilde {\DD}$ is another cell decomposition such that $f$ is cellular for $(\widetilde {\DD}, \DD)$. Then by definition of $\DD'$ every cell in $\widetilde {\DD}$ also lies in $\DD'$.
So we have $\widetilde {\DD}\sub \DD'$. If this inclusion were strict, then there would be a cell $\tau\in \DD'$ whose interior $\inte(\tau)\ne \emptyset$ is  disjoint from the interiors of all cells in 
$\widetilde {\DD}$. This is impossible, since these interiors form a cover of $S^2$. Hence $\widetilde {\DD}=\DD'$. 
\end{proof}

 
 

We now collect properties of the 
cell decompositions $\DD^n=\DD^n(f,\CC)$ from
Defi\-nition~\ref{def:DDn}. 
 
\begin{prop} 
  \label{prop:celldecomp}
 Let $k,n\in \N_0$, 
 $f\: S^2\ra S^2$ be a Thurston map,  $\CC\sub S^2$ be a Jordan curve with $\post(f)\sub \CC$, $\DD^n=\DD^n(f,\CC)$, and   $m=\#\post(f)$. 
 Then the following statements are true:
 
\smallskip
\begin{enumerate}

\item
  \label{item:fkcellular}
  The map  $f^k$ is cellular for $(\DD^{n+k}, \DD^n)$. In
  particular, if  $\tau$ is any $(n+k)$-cell, then $f^k(\tau)$ is an
  $n$-cell, and $f^k|\tau$ is a homeomorphism of $\tau$ onto
  $f^k(\tau)$.  

\item
  \label{item:fkunioncells}
  Let  $\sigma$ be  an $n$-cell. Then $f^{-k}(\sigma)$ is
  equal to the union of all $(n+k)$-cells $\tau$ with
  $f^k(\tau)=\sigma$.  

\item
  \label{item:skeletons}
  The  $0$-skeleton (i.e., the set of vertices) of
  the cell decomposition
  $\DD^n$ 
  is given by  
  ${\bf V}^n=f^{-n}(\post(f))$, and we have ${\bf V}^n \sub {\bf
    V}^{n+k}$.  The $1$-skeleton of $\DD^n$ is  equal to
  $f^{-n}(\CC)$.  

\item
  \label{item:noVEX}
  We have $\#{\bf V}^n\le m \deg(f)^n$,  $\#\E^n=m\deg(f)^n$,  and
  $\#\X^n=2\deg(f)^n$.  

\item
  \label{item:nedgesC}
  The $n$-edges are precisely the closures of the connected
  components of $f^{-n}(\CC)\setminus f^{-n}(\post(f))$. The $n$-tiles
  are precisely the closures of the connected components of
  $S^2\setminus f^{-n}(\CC)$.  

\item
  \label{item:tilesmgon}
  Every $n$-tile  is an $m$-gon, i.e., the number of $n$-edges and
  $n$-vertices contained in  its boundary is equal to $m$.   
  
  \item
  \label{item:celdecompiter} Let $F=f^k$ be an iterate of $f$ with $k\ge 1$. Then 
  $\DD^n(F,\CC)=\DD^{nk}$.
  \index{Thurston map!iterate of}
  \index{iterate of Thurston map}
  \index{F f@$F=f^n$}

\end{enumerate}
\end{prop}

In the proof we will use the following fact about open 
maps $g\: S^2\ra S^2$ (such as iterates of Thurston maps): if $A\sub S^2$ is arbitrary, then
 \begin{equation}\label{eq:openspec}
 g^{-1}(\overline {A})\sub \overline{g^{-1}(A)}.
 \end{equation} 
 Indeed, if $U$ is an open neighborhood of a point 
 $p\in g^{-1}(\overline {A})$, then $g(U)$ is an open neighborhood of $g(p)\in \overline A$. Hence there exists a point $p'\in U$ with $g(p')\in A$, and so $U\cap g^{-1}(A)\ne \emptyset$. The inclusion
 \eqref{eq:openspec} follows.  Note that the reverse inclusion in  \eqref{eq:openspec}  is true for all continuous maps $g$.

\begin{proof} We know that ${\bf V}^0=\post(f)\supset f(\crit(f))$ is the set of vertices of $\DD^0$ and that 
 $f$ is cellular for $(\DD^{n+1}, \DD^n)$ for each $n\in \N_0$. As we have seen in the proof of Lemma~\ref{lem:pullback}, this implies ${\bf V}^{n+1}=f^{-1}({\bf V}^n)$.
 It follows by induction that ${\bf V}^n=f^{-n}(\post(f))$ for $n\in \N_0$. After this preliminary remark, we now turn to the proofs of the statements.   

\smallskip
 \ref{item:fkcellular} This immediately follows from the facts that $f$ is cellular for  $(\DD^{n+1}, \DD^n)$ for each $n$, and that compositions of cellular maps are cellular (if, as in our case,  the obvious compatibility requirement for the cell decompositions involved  is satisfied). 

\smallskip{}
\ref{item:fkunioncells} Note that the set ${\bf V}^n =f^{-n}(\post(f))\supset \post(f)$ of vertices of $\DD^n$ contains the critical values $f^k(\crit(f^k))\sub \post(f)$ of $f^k$.
So we can apply Lemma~\ref{lem:pullback} and conclude from 
\ref{item:fkcellular}  
that $\DD^{n+k}$ is the unique cell decomposition of $S^2$ such that 
$f^k$ is cellular for  $(\DD^{n+k}, \DD^n)$. 
Moreover, recall from the proof of  
Lemma~\ref{lem:pullback} that a topological cell $c\sub S^2$ is an $(n+k)$-cell if and only if $f^k(c)$ is an $n$-cell and $f^k|c$ is a homeomorphism of $c$ onto $f^k(c)$. 

This immediately implies the statement if $\sigma=\{q\}$, where $q$ is an $n$-vertex.  

Suppose $\sigma$ is equal to an $n$-edge $e$. 
Let $M$ be the union of all $(n+k)$-edges $e'$ with $f^k(e')= e$. It is clear that $M\sub  f^{-k}(e)$. 

 To see the reverse inclusion, first note that because there are only finitely many 
 $(n+k)$-edges, the set $M$ is closed.  
 
Let $p\in f^{-k}(\inte(e))$ be arbitrary. Then from Claim 2 in the proof of Lemma~\ref{lem:pullback} it follows that there exists an $(n+k)$-edge $e'$ with $p\in e'$. Then $f^k(e')$ is an $n$-edge that contains $q=f^k(p)\in \inte(e)$. Hence $e=f^k(e')$, and so 
$f^{-k}(\inte(e))\sub M$. 
Since $f^k$ is an open and continuous map and $M$ is closed, it follows from 
\eqref{eq:openspec} that 
$$f^{-k}(e)=f^{-k}(\overline{\inte(e)})\sub \overline { f^{-k}(\inte(e))}\sub \overline M=M.$$
Hence   $M=f^{-k}(e)$ as desired. 

If $\sigma$ is  equal to an $n$-tile $X$, let $M$  be 
the union of all $(n+k)$-tiles $X'$ with $f^k(X')=X$.
Then $M\sub f^{-k}(X)$ and $M$ is closed. 

If $p\in f^{-k}(\inte(X))$, then by Claim 1 in the proof of Lemma~\ref{lem:pullback} there exists an $(n+k)$-tile with 
$p\in X'$. As above, we conclude $f^k(X')=X$, and so $p\in M$. Hence 
$f^{-k}(\inte(X))\sub M$. Now again by \eqref{eq:openspec}  we have 
$$ f^{-k}(X)=f^{-k}(\overline{\inte(X)})\sub\overline{ f^{-k}(\inte(X))}\sub \overline M=M. $$ 
We conclude that $M=f^{-k}(X)$ as desired. 

\smallskip{}
\ref{item:skeletons}
The $0$-skeleton of $\DD^n$ is the set ${\bf V}^n$ of all vertices of $\DD^n$. We have already seen that  ${\bf V}^n=
f^{-n}(\post(f))$. Moreover, 
$$f^{n+k}({\bf V}^n)= f^{n+k}(f^{-n}(\post(f)))\sub f^k(\post(f))\sub \post(f), $$
and so ${\bf V}^n \subset f^{-n-k}(\post(f))= \V^{n+k}$.

The $1$-skeleton of $\DD^n$ is equal to the  set consisting of
all $n$-vertices and the union of all $n$-edges. 
As follows from \ref{item:fkunioncells}, 
this set is equal to the preimage of the $1$-skeleton 
of $\DD^0$ under the map $f^n$. Since the $1$-skeleton 
of $\DD^0$ is equal  to $\CC$,  it follows that the 
$1$-skeleton of $\DD^n$ is equal to $f^{-n}(\CC)$. 

\smallskip{}
\ref{item:noVEX}
Note that $f^n$ is cellular for $(\DD^n, \DD^0)$. Moreover, $\deg(f^n)=\deg(f)^n$, $\#{\bf V}^0=m$, $\#\E^0=m$, and $\#\X^0=2$. The statements about ${\bf V}^n,
\E^n$, and $\X^n$, then follow from the corresponding statement established in the last part of the proof of Lemma~\ref{lem:pullback}. 

\smallskip{}
\ref{item:nedgesC}
This immediately follows from \ref{item:skeletons} and Lemma~\ref{lem:opencells}. 

\smallskip{}
\ref{item:tilesmgon}
If $X$ is an $n$-tile, then $f^n|X$ is a homeomorphism of $X$ onto the
$0$-tile $f^n(X)$. The $n$-vertices contained in $X$ are precisely the
preimages of the $0$-vertices contained in $f^n(X)$; hence $X$
contains exactly $m=\#\post(f)$ $n$-vertices, and hence also the same
number of $n$-edges
(Lemma~\ref{lem:specprop}~\ref{item:prop_cell3}). So every $n$-tile is
an $m$-gon.

\smallskip{} 
\ref{item:celdecompiter} We know that $F=f^k$ is a Thurston map with 
$\post(F)=\post(f)$ (see Section~\ref{sec:defin-thurst-maps}).  
It follows that $\DD^0(F,\CC)=\DD^0$ and that every point in $F^n(\crit(F^n))\sub \post(F)=\post(f)$ is a vertex of $\DD^0(F,\CC)$. By \ref{item:fkcellular} the map  $F^n=f^{nk}$ 
is cellular for $(\DD^{n}(F,\CC), \DD^0(F,\CC))$ and also cellular for  $(\DD^{nk}, \DD^0(F,\CC))$.
Hence $\DD^{n}(F,\CC)=\DD^{nk}$ by the uniqueness statement in Lemma~\ref{lem:pullback}. 
\end{proof}

Instead of an inequality for $\#{\bf V}^n$ as in \ref{item:noVEX} one
can easily give a precise  
formula for this number; namely, if we 
set $d=\deg(f)$ and 
$m=\#\post(f)$, 
then $\#\X^n=2d^n$ and $\E^n=md^n$. Moreover, by Euler's polyhedral formula we have 
$$ \#\X^n-\#\E^n+\#{\bf V}^n=2, $$
and so 
$$ \#{\bf V}^n=  (m-2) d^n+2. $$ 

We record another lemma that relates cells with the mapping properties of a given Thurston map.

\begin{lemma}
  \label{adhoc} 
  Let $k,n\in \N_0$, $f\: S^2 \ra S^2$ be a Thurston map, and
  $\CC\sub S^2$ be a Jordan curve with $\post(f)\sub \CC$.

 \begin{enumerate}
 
 \item
   \label{item:adhoc1}
   If $c\sub S^2$ is a topological cell such that $f^k|c$ is a
   homeomorphism onto its image and $f^k(c)$ is an $n$-cell, then
   $c$ is an $(n+k)$-cell.
  
 \item
   \label{item:adhoc2}
   If $X$ is an $n$-tile and $p\in S^2$ is a point with
   $f^k(p)\in \inte(X)$, then there exists a unique $(n+k)$-tile
   $X'$ with $p\in X'$ and $f^k(X')=X$.
  \end{enumerate} 
\end{lemma} 
 
Here it is understood that all $m$-cells, $m\in \N_0$, are for
$(f,\CC)$.  A priori this is not true for the cell $c$; the point
of \ref{item:adhoc1} is to give a criterion when $c$ is a cell
for $(f,\CC)$ (of appropriate level).
 
\begin{proof}  
  Let $\DD^m=\DD^m(f,\CC)$ for all $m\in \N_0$ according to
  Definition~\ref{def:DDn}.  Note that $f^k$ is cellular for
  $(\DD^{n+k}, \DD^{n})$, and the set
  $f^{-n}(\post(f))\supset \post(f)$ of vertices of $\DD^{n}$
  contains the set $f^k(\crit(f^k))\sub \post(f)$ (the last
  inclusion follows from \eqref{eq:critpfn2}). Hence we are in the
  situation of Lemma~\ref{lem:pullback} with $\DD=\DD^{n}$
  and $\DD'=\DD^{n+k}$.
  
  Then \ref{item:adhoc1} follows from the uniqueness statement of
  Lemma~\ref{lem:pullback} and the definition of $\DD'$ in the
  first paragraph of the proof of this lemma.

  Moreover, under the assumptions of \ref{item:adhoc2} it follows
  from Claim 1 in the proof of Lemma~\ref{lem:pullback} that
  there exists a unique $(n+k)$-tile $X'$ with $p\in X'$. Then
  $f^k(X')$ is an $n$-tile containing $f^k(p)\in \inte(X)$, and
  so $f^k(X')=X$.
\end{proof}

Thurston maps $f\: S^2\ra S^2$ with exactly two postcritical points and their associated cell decompositions $\DD^n(f,\CC)$ are  very special as the next lemma shows.  
Later we will see that   every Thurston map with $\#\post(f)=2$ is in fact  Thurston equivalent to the
map $z\mapsto z^n$ on $\CDach$, where $n\in
\Z\setminus\{-1,0,1\}$ (see Proposition~\ref{prop:post2}). 

\begin{lemma}
  \label{lem:postf-2}
  Let $f\colon S^2\to S^2$ be a Thurston map with precisely  two
  postcritical points $p,q\in S^2$. Let $\CC \subset S^2$ be a
  Jordan curve with $\post(f) = \{p,q\} \subset \CC$ and consider cells for $(f,\CC)$. Then 
  all  $n$-tiles and
  $n$-edges contain $p,q$, and 
  \begin{equation} \label{eq:pqvn2}
    \V^n =f^{-n}(\post(f))  = \post(f) = \{p,q\}
  \end{equation}
  for all $n\in \N_0$. 
\end{lemma}

\begin{proof}
  Let $\alpha_f\colon S^2\to \widehat{\N}$ be the ramification
  function of $f$, and $\mathcal{O}_f=(S^2,\alpha_f)$ be the
  orbifold associated with $f$ (see Definitions~\ref{def:weightf}
  and~\ref{def:orbifold_f}). By Corollary~\ref{cor:post012} we
  know that the signature of $\mathcal{O}_f$ is
  $(\infty,\infty)$. This means that $\alpha_f(p)=
  \alpha_f(q)=\infty$ and $\alpha_f(u)=1$ for all $u\in
  S^2\setminus\{p,q\}$ (see
  Proposition~\ref{prop:otherramprops}~\ref{item:rami_postf}).   

  In particular, $\mathcal{O}_f$ is parabolic. This in turn
  implies that $\alpha_f(u)= \infty$ for all
  $u\in f^{-1}(\post(f))$ (see
  Proposition~\ref{prop:parabolicOf}~\ref{item:Of_para3}). Therefore, 
  $f^{-1}(\post(f)) \subset \post(f)$. Since
  $f^{-1}(\post(f)) \supset \post(f)$ is true for every Thurston
  map (see
  Proposition~\ref{prop:celldecomp}~\ref{item:skeletons}), we
  conclude that $f^{-1}(\post(f)) = \post(f)$. Now
  \eqref{eq:pqvn2} follows by induction.
  
  If $X$ is an $n$-tile, then it   contains two  distinct
  $n$-vertices (see
  Proposition \ref{prop:celldecomp}~\ref{item:tilesmgon}), and so $p,q\in X$.  Similarly, every
  $n$-edge $e$ contains two distinct $n$-vertices (see
  Proposition~\ref{prop:celldecomp}~\ref{item:fkcellular}), and so  $p,q\in e$.   
\end{proof}



Tiles can be used to create connections between sets and points. To make this precise,  we will now introduce various notions of \emph{chains}\index{chain}. For an illustration of the following definitions see  
Figure~\ref{fig:chains}.

\begin{definition}[Chains]
  \label{def:chains}
  A {\em chain} $P$ in $S^2$ is a finite sequence
  $A_1,\dots, A_N$ of sets in $S^2$ such that
  $A_i\cap A_{i+1}\ne \emptyset$ for $i=1, \dots, N-1$. We 
  call $N=\length(P)$ the \emph{length} of the 
  chain.\index{chain!length of}\index{length!of chain} 
  It 
  {\em joins}\index{chain!joins} 
  two sets $A$ and $B$ in $S^2$, if
  $A\cap A_1\neq \emptyset$ and $B\cap A_N\neq
  \emptyset$.  Similarly, $P$  \emph{joins} the points $x,y\in S^2$ if
  $x\in A_1$ and $y\in A_N$ (so $P$ joins 
   $\{x\}$ and $\{y\}$).

    A 
  \emph{subchain}\index{subchain} 
  $P'$ of $P$ is
  a chain  $A_{i_1}, \dots , A_{i_M}$ with 
  $1\leq i_1 < i_2 < \dots <{i_M} \leq N$. We say that the chain $P$ joining the sets $A$ and $B$ 
  is {\em simple} if there is no
  proper subchain $P'$ of $P$ that joins $A$ and $B$. Clearly, if a chain $P$ joins the sets $A$ and $B$, then there is a simple subchain $P'$ of $P$ that
  joins $A$ and $B$. The same terminology   and a similar remark apply to chains joining  two points $x,y\in S^2$.

 We will often consider chains  $X_1,\dots ,X_N$,  where the sets $X_i$ are 
  tiles in  a given cell decomposition
   of $S^2$. In this case, the chain  is called a 
  \emph{chain of tiles}\index{chain!of tiles} 
  or a 
  \emph{tile chain}.\index{tile!chain} 
  
  If each set  $X_i$ of the chain is an  $n$-tile for a given Thurston map 
  $f\:S^2\ra S^2$ and 
   a Jordan curve  $\CC\subset S^2$ with
  $\post(f)\subset \CC$, then we call the chain (with $(f,\CC)$ understood)
  a
  \emph{chain of $n$-tiles}\index{chain!of $n$-tiles} 
  or an \emph{$n$-chain}.\index{n9@$n$-!chain}
  \end{definition}

\begin{definition}[$e$-chains]
  \label{def:e-chain}
 Let $\DD$ be a cell decomposition of $S^2$ and $P$ be a tile chain consisting  
  of the tiles $X_1,\dots,X_N$ in $\DD$. If $P$  has the additional  property that for
  each  $i=1,\dots, N-1$ we have $X_i\neq X_{i+1}$ and there is an
  edge $e_i$ in $\DD$ with
  $e_i\subset \partial X_i \cap \partial X_{i+1}$, then 
  $P$ 
  is called an 
  \emph{$e$-chain}.
  \index{e-chain@$e$-chain|textbf}
  \index{chain!$e$-|textbf}
\end{definition}

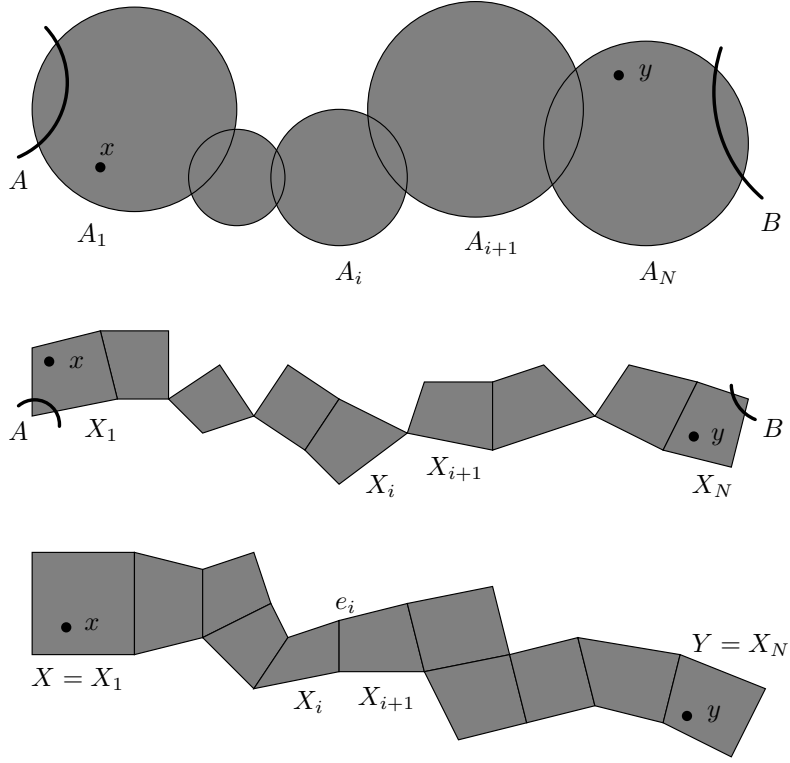
\begin{figure}
  \centering
  \begin{overpic}
    [width=10cm, tics=20,
    ]
    {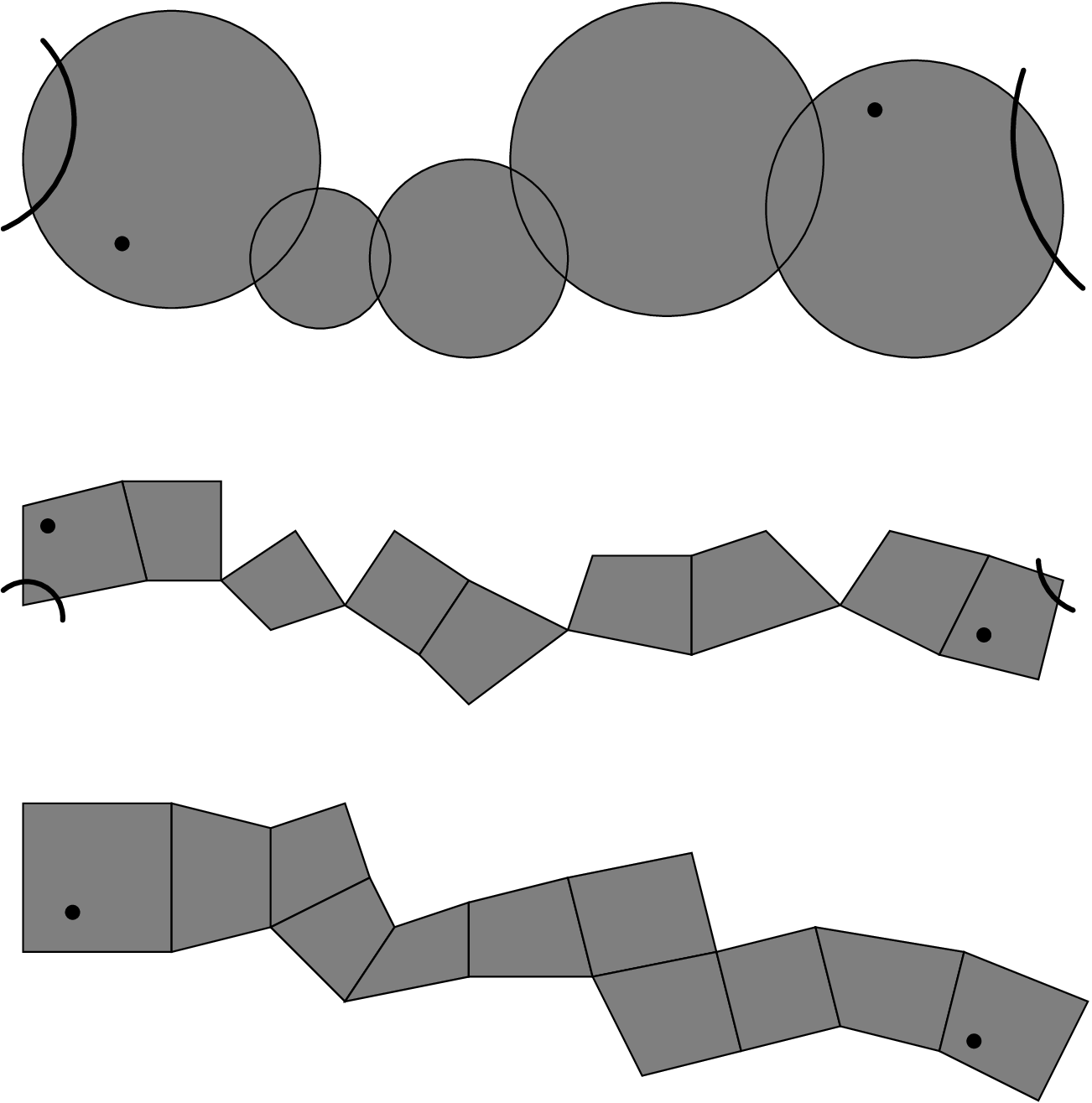}
    \put(8,68){$A_1$}
    \put(82,63){$A_N$}
    \put(-1,75){$A$}
    \put(98,69.5){$B$}
    \put(11,79.5){$x$}
    \put(82,90){$y$}
    \put(42,63){$A_i$}
    \put(59,67){$A_{i+1}$}
    \put(9,42.4){$X_1$}
    \put(89,35){$X_N$}
    \put(46,35){$X_i$}
    \put(54,37.4){$X_{i+1}$}
    \put(-1,42){$A$}
    \put(98.3,42.4){$B$}
    \put(7,51.4){$x$}
    \put(91.5,42){$y$}
    \put(2,9.5){$X=X_1$}
    \put(89,14){$Y=X_N$}
    \put(36.5,6.5){$X_i$}
    \put(45,7){$X_{i+1}$}
    \put(42,19.7){$e_i$}
    \put(9,17){$x$}
    \put(91,5.5){$y$}
  \end{overpic}
  \caption{A chain, an $n$-chain, and an $e$-chain.}
\label{fig:chains}
\end{figure}

With the 
given vertices and edges the $1$-skeleton of $\DD$ can be considered as a  graph embedded in  $S^2$. 
 Then the {\em dual graph} has the set
of tiles as vertices, and two vertices as represented by tiles
are joined by an edge if the tiles both contain an edge
$e\in\DD$ in their boundaries. Then an $e$-chain in $\DD$ is
essentially a path in this dual graph. Having this
interpretation in mind, we say that the $e$-chain $P$ given by $X_1,\dots,X_N$
\emph{joins} the tiles $X=X_1$ and $Y=X_N$. Note that tiles
$X'\neq X$ and $Y'\neq Y$ are not joined by this $e$-chain
according to our definition. If $1\leq i \leq j\leq N$, then
$X_i, \dots, X_j$ is a subchain of $P$; it  is an
$e$-chain that  joins $X_i$ and $X_j$. 

An $e$-chain for a given 
   Thurston map  $f\:S^2\ra S^2$ and   a Jordan curve   
  $\CC\subset S^2$ with 
  $\post(f)\subset \CC$, is an $e$-chain in one of the cell decompositions $\DD=\DD^n(f,\CC)$, $n\in \N_0$. In particular, the tiles in such an $e$-chain  are of the same level $n$.

If $X$ is an arbitrary tile in a cell decomposition $\DD$ of $S^2$, then every tile $Y$ in
$\DD$ can be joined to $X$ by an $e$-chain. This follows from
the fact that the union of the tiles $Y$ that can be joined to
$X$ is equal to $S^2$; indeed, this union is a non-empty closed
set;  it is also open, as follows from
Lemma~\ref{lem:specprop} \ref{item:prop_cell4} and
\ref{item:prop_cell5}. Hence the union is all of $S^2$.
In general, we call a set $M$ of tiles 
{\em $e$-connected}
\index{e-connected@$e$-connected|textbf} 
if every two tiles in $M$ can be joined by an $e$-chain
consisting of tiles in $M$.

Let $f\: S^2\ra S^2$ be a Thurston map, and $\CC\sub S^2$ be a Jordan curve with  $\post(f)\sub \CC$. 
It is often useful, in particular in graphical representations, to assign to each tile in $\DD^n(f,\CC)$ 
one of the two 
colors ``black'' or ``white'' 
represented by the symbols 
${\tt b}$ and ${\tt w}$, respectively.   To formulate this, we denote 
by $\X^\infty$ the disjoint union of the sets $\X^n=\X^n(f,\CC)$, 
$n\in \N_0$. 
More informally, $\X^\infty$ is the set of all tiles for $(f,\CC)$. Note that 
in general   a set can be a tile  for different levels $n$, so  the same tile  may  be represented by multiple copies  in $\X^\infty$ distinguished 
by their  levels.

\index{colors of tiles}\index{tile!color of}\index{b@${\tt b,w}$}

\begin{lemma}[Colors of tiles]\label{lem:colortiles}
There exists a map $L\: \X^\infty\ra \{{\tt b}, {\tt w}\} $ with the following properties: 

\begin{enumerate}

\item
\label{item::colortiles1}
  $L(X^0_{\tt b})={\tt b}$ and $L(X^0_{\tt w})={\tt w}$.  

\item
\label{item::colortiles2}
  If $n,k\in \N_0$,  $X^{n+k}\in \X^{n+k}$,  and $X^n=f^k(X^{n+k})
  \in \X^n$, then $L(X^{n})=L(X^{n+k})$.

\item
\label{item::colortiles3}
  If $n\in \N_0$, and $X^n$ and $Y^n$ are two distinct $n$-tiles 
  that have an $n$-edge in common, then $L(X^{n})\ne L(Y^{n})$. 
\end{enumerate}
 
 Moreover, $L$ is uniquely determined by properties
 \ref{item::colortiles1} and \ref{item::colortiles2}.  
\end{lemma}

So with the normalization \ref{item::colortiles1} one can uniquely assign colors ``black'' or  ``white'' to the tiles so that all iterates of $f$ are color-preserving as in \ref{item::colortiles2}. 
By \ref{item::colortiles3} colors of distinct $n$-tiles are different if they share an $n$-edge. 

To find the number of white $n$-tiles for $n\in \N_0$,  pick a point  $p\in \inte(X^0_{\tt w})\sub S^2\setminus \post(f)$. Then $p$ is not a critical value of $f^n$, and so 
 $\#f^{-n}(p)=\deg(f)^n$. On the other hand, it follows from 
 Proposition~\ref{prop:celldecomp}~\ref{item:fkcellular} that each white $n$-tile $X^n$ contains a unique point $q\in f^{-n}(p)$. It  lies in the interior of $X^n$, and so for different $n$-tiles these points are
  distinct; moreover,   by Lemma~\ref{adhoc}~\ref{item:adhoc2}   each $q\in \#f^{-n}(p)$ is contained in a unique white $n$-tile. So we have a bijection between the set of white $n$-tiles and $f^{-n}(p)$.  Hence the number of white $n$-tiles is equal to $\#f^{-n}(p)=\deg(f)^n$. A similar argument shows  that the number of black $n$-tiles is also equal to 
   $\deg(f)^n$.

Our notion of colorings of tiles is related to the more general concept 
of a {\em labeling} of cells in a cell decomposition (see  Section~\ref{sec:labelings},  and  in particular Lemma~\ref{lem:labelexis}). 

\begin{proof}[Proof of Lemma~\ref{lem:colortiles}] 
  To define $L$ we assign colors to the two $0$-tiles
  $X^0_{\tt b}$ and $X^0_{\tt w}$ as in \ref{item::colortiles1}. If
  $Z^n$ is an $n$-tile 
  for some arbitrary level $n\ge 0$, then $f^n(Z^n)$ is a $0$-tile
  (Proposition~\ref{prop:celldecomp}~\ref{item:fkcellular}), and so it
  already has a 
  color assigned. We set $L(Z^n)\coloneqq L(f^n(Z^n))$.
 
  This defines a map $L\: \X^\infty\ra \{{\tt b}, {\tt w}\}$. By
  definition, $L$ has property \ref{item::colortiles1}. To show
  \ref{item::colortiles2}, assume that $n,k\in \N_0$ and $X^{n+k}\in
  \X^{n+k}$. Then by Proposition~\ref{prop:celldecomp}~\ref{item:fkcellular}, we have
  $X^n\coloneqq f^k(X^{n+k})\in \X^n$, and $f^{n+k}(X^{n+k}),f^n(X^n)\in
  \X^0$. So by definition of $L$ we have
  \begin{align*}
  L(X^n)&=L(f^n(X^n))=L(f^n(f^k(X^{n+k})))\\
   &=L(f^{n+k}(X^{n+k}))=L(X^{n+k})
   \end{align*} 
  as desired. 
  
  Let $X^n$ and $Y^n$ be as in \ref{item::colortiles3}. Then again by
  Proposition~\ref{prop:celldecomp}~\ref{item:fkcellular}, we have
  $f^n(X^n), 
  f^n(Y^n)\in \X^0$. There exists an $n$-edge $e$ such that  $e\sub \partial X^n\cap \partial Y^n$.  We orient $e$ so that $X^n$ lies to the left of $e$. The set of $n$-vertices is equal to $f^{-n}(\post(f))$ and disjoint from $\inte(e)$. It follows that no point in $\inte(e)$ is a critical point of $f^n$. In particular,  $f^n$ is a local homeomorphism near each point in $\inte(e)$
  which implies that $f^n(X^n)$ and $f^n(Y^n)$ are distinct.    
 We conclude that   $L(f^n(X^n))\ne L(f^n(Y^n))$, and so by definition of $L$ we have
  $$ L(X^n)=L(f^n(X^n))\ne L(f^n(Y^n))=L(Y^n)$$ 
  as desired. It follows that $L$ has the properties
  \ref{item::colortiles1}--\ref{item::colortiles3}. 
  
  It is clear that $L$ is uniquely determined by
  \ref{item::colortiles1} and \ref{item::colortiles2}.  
\end{proof} 

If the tiles in a cell decomposition $\DD$ of a $2$-sphere $S^2$
are assigned 
colors ``black'' or ``white'' so that 
two distinct
tiles sharing an edge have different colors, then we say the
cell decomposition is a 
{\em checkerboard tiling}\index{checkerboard tiling} 
of $S^2$. 
It is clear that for the existence of such a coloring the length of the cycle of each vertex in $\DD$ has to be even. In Lemma~\ref{lem:labelexis} we will see that this necessary condition is also sufficient. 

If there exists $m\in \N$, $m\ge 2$, such that each tile in $\DD$
is an $m$-gon, then we say that $\DD$ is a 
{\em tiling by $m$-gons}.\index{tiling by $m$-gons} 
With this terminology we can summarize some of the main results of this section by saying that the cell decompositions $\DD^n=\DD^n(f,\CC)$ (with the colorings given by the previous lemma) are checkerboard tilings by $m$-gons, where $m=\#\post(f)$.

 \section{Labelings}
\label{sec:labelings}
 Suppose $\DD^0$ and $\DD^1$ are cell decompositions of a $2$-sphere
$S^2$.  In Section~\ref{sec:mapsfromdecomp} we will see that under suitable conditions one can construct a Thurston
map that is cellular for $(\DD^1,\DD^0)$. If one wants to obtain a unique map up to Thurston equivalence,  one needs additional   data;  namely, for each cell in $\DD^1$ we have to assign  an  image in $\DD^0$. 
The necessary properties of such assignments can be abstracted in the notion of a {\em labeling}.

 \begin{definition}[Labelings]\label{def:labeldecomp}
  Let  $\DD^1$ and $\DD^0$ be cell complexes.   Then  a  {\em
    labeling}\index{labeling}  
  of $(\DD^1, \DD^0)$ is a map 
  $L\: \DD^1 \ra \DD^0$ satisfying the following conditions: 
  
  \begin{enumerate} 
  
  \item
    \label{item:labeldecomp1}
    $\dim (L(\tau))=\dim(\tau)$ for all $\tau\in \DD^1$.
  
  \item
    \label{item:labeldecomp2}
  If $\sigma, \tau\in \DD^1$ and $\sigma\sub \tau$, then $L(\sigma)\sub L(\tau)$.
  
  \item
    \label{item:labeldecomp3}
    If $\sigma, \tau,c\in \DD^1$, $\sigma,\tau\sub c$, and $L(\sigma)=L(\tau)$, then $\sigma=\tau$.
  \end{enumerate} 
\end{definition}

So a labeling is a map $L\: \DD^1\ra \DD^0$ that preserves inclusions
and dimensions of cells, and is ``injective on cells'' $c\in \DD^1$ in
the sense of \ref{item:labeldecomp3}.  
In particular, every  cell of dimension $0$ in $\DD^1$ is mapped
to  a cell of dimension $0$ in $\DD^0$. 
If $v$ is a vertex  
in $\DD^1$, i.e., if  $\{v\}$ is a cell of dimension $0$ in $\DD^1$,  then  we can write $L(\{v\})=\{w\}$, where 
$w$ is a vertex in $\DD^0$. We define $L(v)=w$. In the following, we always assume that a labeling $L\: \DD^1\ra \DD^0$ has been extended to the set of 
vertices of $\DD^1$ in this way; this will allow us to ignore the  distinction between vertices and cells of dimension $0$, i.e., sets consisting of one vertex.

 Let $S^2$ be an oriented $2$-sphere, and   
  $\DD$ be  a cell decomposition of $S^2$. Recall (see Section~\ref{sec:2spherecd}) that  a flag in $\DD$  is a triple $(c_0,c_1,c_2)$, where $c_i$ is a cell in $\DD$ of dimension  $i$ for $i=0,1,2$ and $c_0\sub c_1 \sub c_2$.   If $L\:\DD^1\ra \DD^0$ is a labeling of a pair $(\DD^1, \DD^0)$ of cell decompositions of $S^2$ and $(c_0,c_1,c_2)$  is a flag in 
  $\DD^1$, then $(L(c_0),L(c_1),L(c_2))$ is a flag in $\DD^0$. This   follows from the definition of a labeling. So a labeling maps 
  ``flags to flags''. We say that  the labeling is 
{\em orientation-preserving}\index{labeling!orientation-preserving}\index{orientation!preserving!labeling}  
  if it maps  flags in $\DD^1$ to 
  flags in $\DD^0$ of the same (positive or negative)  orientation. Here the orientation of a flag is determined by the given orientation of the underlying $2$-sphere $S^2$.

 If $f\:S^2\ra S^2$ is  cellular for $(\DD^1, \DD^0)$, then the
 map $L\:\DD^1\ra \DD^0$ given by  $L(\tau)=f(\tau)$ for  $\tau \in \DD^1$ is a labeling. It is called the 
 \emph{labeling induced by $f$}\index{labeling!induced by $f$}. 
 If a Thurston map $f\: S^2 \ra S^2$  is  cellular for $(\DD^1, \DD^0)$, then its induced labeling $L\:\DD^1\ra \DD^0$ is orientation-preserving. This follows from the fact that  
  on $S^2\setminus \crit(f)$ the map $f$ is an orientation-preserving local homeomorphism. So for each tile  $X\in \DD^1$ the homeomorphism $f|X$ must be  
orientation-preserving   in the sense that $f$ preserves the orientation of flags contained in $X$. 
 
 If a labeling $L\:\DD^1\ra \DD^0$ is given, then we say that a map
 $f\:S^2\ra S^2$ that is cellular for $(\DD^1, \DD^0)$ is {\em
   compatible}\index{labeling!compatible} with the labeling $L$ if
 $L(\tau)=f(\tau)$ for each $\tau \in \DD^1$, i.e., if the labeling
 induced by $f$ is equal to the given labeling.

 Let $X$ be a closed Jordan region in the oriented $2$-sphere $S^2$
 with $k\ge 3$ distinct points $v_0, \dots, v_{k-1}, v_{k}=v_0$ on $\partial X$. Here
 the indices are elements of $\Z_k=\{0, 1,\dots,k-1\}=\Z/k\Z$, the
 cyclic group with $k$ elements. Suppose further that the points $v_0,
 \dots, v_{k-1}$ are indexed such that if we start at $v_0$ and run
 through $\partial X$ with suitable orientation, then the points $v_0,
 \dots, v_{k-1}$ are traversed in successive order. If this is true
 and if with this orientation of $\partial X$ the region $X$ lies on
 the left, 
then we call the points $v_0, \dots, v_{k-1}$ 
in 
{\em cyclic} order\index{cyclic order} 
on $\partial X$, and otherwise, if $X$ lies on the right, in
 {\em anti-cyclic} order\index{anti-cyclic order} 
on $\partial X$.  If the points $v_0, \dots,
 v_{k-1}$ are in cyclic or anti-cyclic order on $\partial X$, then
 $\partial X$ is decomposed into unique arcs $e_0, \dots, e_{k-1}$;
 here $e_l$ for $l\in \Z_k$ is the unique subarc of $\partial X$ that
 has the endpoints $v_l$ and $v_{l+1}$, but does not contain any other
 of the points $v_i$, $i\in \Z_k\setminus\{l,l+1\}$. We say that the
 arcs $e_0, \dots, e_{k-1}$ are in {\em cyclic} or {\em anti-cyclic}
 order on $\partial X$, if this is true for the points $v_0, \dots,
 v_{k-1}$, respectively.

  If we have a labeling $L\: \DD^1\ra \DD^0$,  we should think of each element $\tau\in \DD^1$ as ``carrying'' the label $L(\tau)\in \DD^0$. In applications it is often 
 more intuitive and convenient to allow more general index sets $\mathcal{L}$ of the same cardinality  as $\DD^0$ as labeling sets for the elements in $\DD^1$. 
 In such situations we fix a bijection $\psi\: \DD^0\ra \mathcal{L}$ and call a map 
 $L'\: \DD^1 \ra \mathcal{L}$ a labeling if $\psi^{-1}\circ L'\: \DD^1\ra \DD^0$ is a labeling in the sense of Definition~\ref{def:labeldecomp}.

 We will discuss this  in a case that will be relevant for us
 later. Namely, suppose that the cell decomposition $ \DD^0$ of the
 oriented sphere $S^2$ has two tiles  $X^0_{\tt b}$ and $X^0_{\tt w}$ 
 with common boundary $\CC\coloneqq  \partial X^0_{\tt b}= \partial X^0_{\tt w}$. We represent 
$X^0_{\tt b}$ by the symbol  $\tt b$ for ``black'' and $X^0_{\tt w}$ by $\tt w$ for ``white''.
By Lemma~\ref{lem:specprop}~\ref{item:prop_cell3} the set  $\CC$ is a
Jordan curve containing $k\ge 2$ vertices and edges, and there are no
other edges and vertices in $\DD^0$. Let us assume that $k\ge 3$ and that  we have indexed  the vertices $v_0, \dots, v_{k-1}$ so that they are
  cyclically ordered on  $\partial X^0_\wt$. Then they are 
anti-cyclically ordered on $\partial
  X^0_\bt$. As above, we index the edges such that $e_l$ is the unique
  subarc on $\CC=\partial X^0_\wt$ with endpoints $v_l$ and
  $v_{l+1}$.
  
There exists a bijection of  $\DD^0$ with  the  set  
$\mathcal{L}$ consisting of the symbols 
  $\tt b$ and $\tt w$ (for the two tiles in $\DD^0$) and two copies of $\Z_k$, one for the edges and the other one for the vertices in $\DD^0$.  
 More explicitly, such a bijection  $\psi\colon \DD^0\to \mathcal{L}$ is  given 
 by
\begin{equation}
  \psi(X^0_\wt)= \wt, \  \psi(X^0_\bt)=\bt,\   \psi(v_l)= l \text{ and } \psi(e_l)=l
  \text{ for } l\in \Z_k.
\end{equation}

Now suppose that in this situation  $\DD^0$ is equal to 
the cell decomposition $\DD^0(f,\CC)$ (as defined in Section~\ref{sec:tiles}) for a Thurston map $f\: S^2\ra S^2$. In other words, $\post(f)\sub \CC$ and 
the vertex set of $\DD^0$ is equal to $\post(f)$. Let $\DD^1=\DD^1(f,\CC)$ be the cell decomposition given by cells of  level $1$,  and
$\X=\X^1$, $\E=\E^1$, and $\V=\V^1$ be the sets of $1$-tiles,
$1$-edges, and $1$-vertices, respectively. One can then define three maps 
 $L_{\X}\colon \X\to \{\wt, \bt\}$,
$L_{\E}\colon \E\to \Z_k$ and $L_{\V} \colon \V\to \Z_k$  
as \begin{equation*}
  L_{\X}(X) = \psi(f(X)), \quad 
  L_{\E}(e)= \psi(f(e)), \quad
  L_{\V}(v)= \psi(f(v))
\end{equation*}
for $X\in \X$, $e\in \E$, and $v\in \V$.  Accordingly, a $1$-tile $X$ is 
called ``white'' if $L_{\X}(X)= \wt$, and called ``black'' if
$L_{\X}(X)=\bt$.  It follows from
Proposition~\ref{prop:celldecomp}~\ref{item:fkcellular} that for each 
$X\in \X$ the map 
$f|X$ is an orientation-preserving  homeomorphism onto either
$X^0_\wt$ or $X^0_\bt$. This implies that if we use the map $L_{\bf V}\: {\bf V}\ra \Z_k$ to index $1$-vertices, then they are in cyclic order on the boundary $\partial X$ 
of a 
white $1$-tile $X$, and in 
anti-cyclic order 
on the boundary of a black $1$-tile. 
 Similarly, the  $1$-edges are in cyclic order
 on the boundary of white $1$-tiles, and in anti-cyclical order  on the
boundary of black $1$-tiles. 
The maps $L_{\X}$, $L_{\E}$, and $L_{\V}$ 
can be combined in  the obvious way to a map 
$L\: \DD^1\ra \mathcal{L}\cong \DD^0$ (so that $L|{\X}=L_{\X}$, etc.), and it follows 
easily from the previous discussion that $L$ is a labeling.

In the following lemma, we will turn this construction around and ask 
when a labeling with similar properties 
exists on a given cell decomposition $\DD=\DD^1$ of a $2$-sphere that is {\em a prori} not  related to a  Thurston map.     
This will later be useful when we want to construct Thurston maps. 

\begin{lemma}\label{lem:labelexis}
Let $\DD$ be a cell decomposition of $S^2$, and denote by ${\bf V}$
the set of vertices, by $\E$ the set of edges,   and by $\X$ the set of tiles in $\DD$. 
Suppose that the length of the cycle of every  vertex in $\DD$ is even and  that there exists $k\ge 3$ such that every tile in $\X$ is a $k$-gon.

Then for 
each 
positively-oriented flag $(c_0,c_1,c_2)$ in $\DD$
there are  maps 
$L_{{\bf V}}\: {\bf V}\ra \Z_k$, $L_\E\: \E\ra \Z_k$, and $L_\X\: \X\ra \{{\tt b}, {\tt w}\}$  with the following properties: 

\begin{enumerate}

\item
  \label{item:labelexis1}
  $L_{{\bf V}}(c_0)=0$,  $L_\E(c_1)=0$, and  $L_\X(c_2)={\tt w}$. 

\item
  \label{item:labelexis2}
  If $X, Y\in \X$ are two distinct tiles with a common edge on their
  boundaries, then $L_\X(X)\ne L_\X(Y)$. 

\item
  \label{item:labelexis3}
  If $X$ is an arbitrary tile in $\X$, then  $L_{\bf V}$ induces a bijection  of 
  the set of vertices in $\partial X$ with $\Z_k$ so that   the order of these vertices 
  is cyclic if $L_\X(X)={\tt w}$ and anti-cyclic if $L_\X(X)={\tt b}$. 

\item
  \label{item:labelexis4}
  If $e\in \E$  and $l=L_\E(e)$, then $L_{{\bf V}}(\partial e)=
  \{l,l+1\}$.     

\item
  \label{item:labelexis5}
  If $X$ is an arbitrary tile in $\X$, then $L_{\bf E}$ induces a bijection  of the set of edges 
  contained in $\partial X$ 
   so that the order of these edges  is
  cyclic if $L_\X(X)={\tt w}$ and anti-cyclic if $L_\X(X)={\tt b}$.  

\item
  \label{item:labelexis6}
 A flag $(\tau_0, \tau_1, \tau_2)$  in $\DD$  is
  positively-oriented if and only if there exists $l\in \Z_k$ such
  that $(L_{\bf V}(\tau_0), L_\E(\tau_1), L_\X(\tau_2))$ is equal to 
 $(l, l, \tt w)$ or $(l, l-1, \tt b)$.
  \end{enumerate}

The maps $L_{{\bf V}}$, $L_\E$, and $L_\X$ are uniquely determined by
the properties \ref{item:labelexis1}--\ref{item:labelexis4}.  
\end{lemma}

Here in \ref{item:labelexis1} and \ref{item:labelexis6} we again
ignored the distinction 
between vertices 
and $0$-dimensional cells in $\DD$ by setting 
$L_{\bf V}(c)=L_{\bf V}(v)$ if $c=\{v\}$ is a $0$-dimensional cell consisting of the vertex $v$. 
 
Recall (see the discussion after Lemma~\ref{lem:specprop}) that the
length of the cycle of a vertex $v$ is the number of edges as well as
the number of tiles that contain $v$. So instead of saying that the
length of each cycle of every vertex is even, we could have said that
every vertex is contained in an even number of tiles (equivalently
contained in an even number of edges).

Condition \ref{item:labelexis2} says that one of the two tiles
containing an edge is ``black'' and the other is ``white''. So if the
tiles have been labeled in this way, then $\DD$ becomes a checkerboard
tiling by $k$-gons. Condition~\ref{item:labelexis2} is equivalent to
the statement that the dual graph of the $1$-skeleton of $\DD$ is
bipartite. 

Condition \ref{item:labelexis3} says that for any white tile $X$,
the vertices on $\partial X$ are labeled cyclically; for any black
tile $Y$, the vertices on $\partial Y$ are labeled anti-cyclically. A more
precise formulation is as follows: for each $i\in \Z_k$ there is
exactly one vertex $v\in \partial X$ with $L_{\bf V}(v)=i$, and if we
write $v=v_i$ if $L_{\bf V}(v)=i$, then the vertices $v_0, \dots,
v_{k-1}\in \partial X$ are in cyclic order on $\partial X$ if $X$ is
white and in anti-cyclic order if $X$ is black.
Condition~\ref{item:labelexis5} has to be interpreted in a similar
way.

  By \ref{item:labelexis4} the label $L_\E(e)$ of an edge $e\in \E$ is  determined by the labels $L_{\bf V}(u)$ and 
$L_{\bf V}(v)$ of the two endpoints $u$ and $v$  of $e$ (here it is important that $k\ge 3$).



\begin{proof}[Proof of Lemma~\ref{lem:labelexis}] We first establish  the following fact. 

\smallskip 
{\em Claim.} Suppose that  $J\sub S^2$ is a Jordan curve that does not contain
any  vertex (in $\DD$) and has the property that for every edge $e$  
the intersection $e\cap J$ is either empty,  or $e$ meets both  components of $S^2\setminus J$ and $e\cap J$  consists of a single point. Then $J$ meets an even  number of edges.  

\smallskip
To see this, pick one of the complementary components $U$ of $S^2\setminus J$, and
let $v_1, \dots, v_n$ be the vertices contained in $U$, where $n\in \N_0$ (for $n=0$ we consider this as an empty list). For $i=1, \dots, n$ let $d_i$ be the length of the cycle of $v_i$, i.e., the number of edges containing $v_i$. 
We denote by $\E_J$ the set  of  all edges that meet $J$ and  by $\E_U$  the set  of all edges contained in $U$.  From  our assumption on the intersection property of $J$ with edges it follows that an edge is contained in $U$ if and only if its two endpoints are in $U$, and it meets $J$ if and only if one endpoint is in $U$ and the other in $S^2\setminus \overline{U}$. Hence 
$$d_1+\dots+d_n=\#\E_J+2\#\E_U,$$
because the sum on the left hand side counts every edge in
$\E_J$ once, and every edge in $\E_U$ twice.  Since all the
numbers $d_1, \dots, d_n$ are even by our assumptions, we conclude 
 that the number $\#\E_J$ of edges that $J$ meets is also
even. The claim follows.  

\smallskip
We now proceed to show existence and uniqueness of the map
$L_\X$. For every tile $Y$ there exists an $e$-chain 
$Y_0=c_2, \dots, Y_N=Y$ of tiles joining the ``base tile'' $c_2$
(from the given flag $(c_0,c_1,c_2)$) to $Y$. Recall from
Definition~\ref{def:e-chain} that such an $e$-chain is a finite
sequence $Y_0, \dots, Y_N$ of tiles such that $Y_i\ne Y_{i+1}$ and there is an edge $e_i
\subset \partial Y_i \cap \partial Y_{i+1}$ for  $i=0,\dots, N-1$ (note that in contrast to Definition~\ref{def:e-chain} it is convenient to start  the index at $i=0$ here). 
We put $L_\X(Y)={\tt w}$ or $L_\X(Y)={\tt b}$ depending on
whether  
$N$ is even or odd. It is clear that if this is well-defined, then it is the unique choice for $L_\X(Y)$. This follows from the normalization \ref{item:labelexis1} and the fact that by \ref{item:labelexis2} the labels of tiles have to alternate along an $e$-chain.

To see that $L_\X$ is well-defined, it is enough to show that if an  
$e$-chain $X_0, X_1, \dots, X_{N}$   forms  a cycle, i.e., if $X_0=
X_{N}$, then $N$ is even. To prove this, we may make the additional assumption that $N\ge 3$ and that the chain is {\em simple}, i.e., that the tiles $X_0, \dots, X_{N-1}$ are all distinct.   

Let $e_i$ be an  edge with $e_i\sub \partial X_{i} \cap \partial X_{i+1}$ for $i=0, \dots, N-1$.  Then the edges $e_0, \dots, e_{N-1}$ are all distinct. For otherwise, $e_i=e_j$ for some 
$0\le i<j\le N-1$. Then $e_i=e_j$ is contained in the boundary of the  tiles $X_{i}, X_{i+1}, X_{j}, X_{j+1}$ which is impossible, because three of these tiles must be distinct (note that $N\ge 3$).  

We now  construct a Jordan curve $J$ that ``follows'' our closed
$e$-chain. Formally, for each edge $e_i$ pick a point $x_i\in \inte(e_i)$. Moreover, for $i=0,\dots, 
N-1$, we can choose  an arc $\alpha_i\sub X_i$  with endpoints $x_{i}$ and $x_{i+1}$ such
that $\inte(\alpha_i)\sub \inte(X_i)$. Here $x_{N}\coloneqq x_0$. 
Then $J=\alpha_0\cup \dots \cup \alpha_{N-1}$ is a Jordan curve that has  properties as in the  claim above. The curve $J$ meets the edges  $e_0, \dots, e_{N-1}$ and no others. Hence $N$ is even. Thus 
$L_\X$ is well-defined, has  property
\ref{item:labelexis2}, and is  normalized as  in
\ref{item:labelexis1}. 

 \smallskip 
To show the existence of $L_{{\bf V}}$, it is useful to quickly recall some basic definitions from the homology and cohomology of chain complexes.
Denote by $\E_o$  the set of oriented edges in $\DD$. Let 
$C(\X)$ and  $C(\E_o)$ be the free modules   over $\Z_k$ 
generated by the sets $\X$ and  $\E_o$,  respectively.
So $C(\E_o)$, for example, is just the set of formal finite sums 
$\sum a_ie_i$, where $a_i\in \Z_k$ and $e_i\in \E_o$. Note that in contrast to other commonly used definitions of chain complexes we have  $e+\widetilde e\ne 0$ if $e$ and $\widetilde e$ are oriented edges with the same underlying set, but 
 opposite orientations.
 
There is a unique boundary operator $b\: C(\X)\ra C(\E_o)$ 
that is a module homomorphism and satisfies 
$$bX\coloneqq b(X)=\sum_{e\sub \partial X} e $$
for each tile $X$, 
where the sum is extended over all oriented edges $e\sub \partial X$ so that $X$ lies on the left of $e$. 

Let $e$ be an oriented edge and $X$ be the unique tile with $e\sub\partial X$ that is on the left of $e$. We put $\alpha(e)=1\in \Z_k$ or $\alpha(e)=-1\in \Z_k$ depending on whether $L_\X(X)={\tt w}$ ($X$ is a white tile) or $L_\X(X)={\tt b}$ ($X$ is a black  tile). If $e$ and $\widetilde e$ are oriented  edges with the same underlying set, but 
 opposite orientations, then $\alpha(e)+\alpha(\widetilde e)=0$ as follows from property \ref{item:labelexis2} of $L_\X$. 

The map $\alpha$ extends 
uniquely 
to a  homomorphism 
$\alpha\: C(\E_o)\ra \Z_k$.  In the language of cohomology it is a ``cochain''. This cochain $\alpha$ is a cocycle, i.e., 
\begin{equation}\label{abX=0}
\alpha(b X)=\sum_{e\sub\partial X}\alpha(e)=\pm k=0\in \Z_k
\end{equation}
 for every tile $X$,  
considered as one of the  generators of $C(\X)$. 
 Indeed, by our convention on the orientation of edges $e\sub \partial X$ in the above sum, for each such edge we get the same contribution $\alpha(e)$, and so, since $X$ has $k$ edges,  the  sum is equal to $\pm k=0\in \Z_k$.
 
 Consider an arbitrary closed edge path 
that consists
of the oriented edges
 $e_1, \dots, e_n$; so the terminal point of $e_i$ is the initial point of $e_{i+1}$ for $i=1, \dots, n$, where $e_{n+1}\coloneqq e_1$. 
 We claim that  
 \begin{equation}\label{alphacycle}
 \sum_{i=1}^n\alpha(e_i)=0.
 \end{equation}
 Essentially, this is a consequence of the fact that we have 
 $H^1(S^2, \Z_k)=0$ for the first cohomology group of $S^2$ with coefficients in $\Z_k$. This   implies that the cocycle $\alpha$ is a coboundary and gives \eqref{alphacycle}. 
 
 We will present a simple direct argument. To show \eqref{alphacycle}, it is clearly enough to establish this
for simple closed edge paths, i.e., for closed  edge paths  where the underlying sets of 
all edges are distinct and have 
 a Jordan curve $J\sub S^2$  as a union. In this case,  let $U$ be 
  the complementary component of $S^2\setminus J$ so that $U$ lies on the left if we traverse $J$ according to the orientation given by the edges $e_i$.  If $X_1, \dots, X_M$ are all the tiles contained in $\overline U$, then
 $$b(X_1+\dots+X_M)=\sum_{e\sub \overline U}e, $$
where the sum is extended over  oriented edges contained in $\overline U$.  Each  edge on $J$ is equal to one of the edges $e_i$ and it appears in the above sum exactly once and  with the same orientation   as $e_i$. All other edges  in $\overline U$ appear twice and with opposite orientations. 
Hence by \eqref{abX=0}, 
$$\sum_{i=1}^n \alpha(e_i)= \sum_{e\sub \overline U}\alpha(e)=
\sum_{i=1}^M\alpha(b X_i)=0. $$ 
 
 We now define $L_{{\bf V}}\: {\bf V}\ra \Z_k$ as follows. Suppose $c_0=\{p_0\}$ consists of the vertex $p_0$. 
 For $v\in {\bf V}$ we can pick an edge path consisting of the oriented edges $e_1, \dots, e_n$ that joins the base point $p_0$ to $v$ (this list of edges may be empty if $v=p_0$).
 The existence of such an edge path follows from the connectedness of the $1$-skeleton of $\DD$ (see Lemma~\ref{lem:specprop}~\ref{item:prop_cell6}). 
 Put 
 \begin{equation}\label{defphi}
 L_{{\bf V}}(v)=\sum_{i=1}^n\alpha(e_i). 
 \end{equation}
 This is well-defined, because 
 we have \eqref{alphacycle} for every closed edge path; we also  have the normalization $L_{{\bf V}}(c_0)=L_{{\bf V}}(p_0)=0$.

 The definition of $L_{{\bf V}}$ implies that if $e$ is an oriented edge, and $u$ is the initial and $v$ the terminal point of $e$, then
 \begin{equation}\label{incrdecr}
L_{\bf V}(v)=L_{{\bf V}}(u)+\alpha(e).
\end{equation} 
 This means that if we go from the initial point $u$ of $e$ to the terminal point $v$, then the value of $L_{{\bf V}}$ is increased by $1$ or decreased by $-1$ depending on whether  the tile on the left of $e$ is white or black. The desired property \ref{item:labelexis3} of $L_{{\bf V}}$ immediately follows from this. 
 
 This shows existence of $L_{{\bf V}}$. Conversely, every function
 $L_{{\bf V}}$ with property \ref{item:labelexis3} must satisfy
 \eqref{incrdecr}. Together with the normalization $L_{{\bf
     V}}(p_0)=0$ this implies that $L_{{\bf V}}$ is given by the
 formula \eqref{defphi}, and so we have uniqueness.

\smallskip  
  To define $L_\E$  note that if  
$e\in \E$, then by \ref{item:labelexis2} we can choose a unique orientation for $e$ such that the tile on the left is white, and the one on the right is black. If $u$ is the initial and $v$ the terminal point of $e$ according to this orientation, and $L_{{\bf V}}(u)=l\in \Z_k$, then $L_{{\bf V}}(v)=l+1$.
Now set $L_\E(e)\coloneqq  l$.  Then $L_\E$ has property \ref{item:labelexis4}. 
Moreover, we also have the normalization \ref{item:labelexis1} for $L_\E$; indeed, if $c_1$ is oriented so that $p_0$ is the initial point of $c_1$, then $c_2$ lies on the left of $c_1$, because the flag $(c_0,c_1,c_2)$ is positively-oriented. Since $L_\X(c_2)={\tt w}$, the tile  $c_2$ is white and so $L_\E(e)=L_{{\bf V}}(p_0)=0$. Uniqueness of $L_\E$ follows from \ref{item:labelexis3} and the uniqueness of $L_{{\bf V}}$. 

We have proved \ref{item:labelexis1}--\ref{item:labelexis4} and the uniqueness statement. It remains to establish \ref{item:labelexis5} and \ref{item:labelexis6}. 

\smallskip
To show \ref{item:labelexis5} let $X\in \X$ be arbitrary. Then by \ref{item:labelexis3} we can assume that the indexing of the $k$ vertices 
$v_0, \dots, v_{k-1}$ on $\partial X$ is such that $L_{\bf V}(v_i)=i$ for all $i\in \Z_k$, and that 
$v_0, \dots, v_{k-1}$ are met in successive order  if we traverse $\partial X$. This implies 
that for each $i\in  \Z_k$ there exists a unique edge  $e_i\sub \partial X$  in $\DD$ with endpoints $v_i$ and $v_{i+1}$. Hence by \ref{item:labelexis4} we have $L_{\E}(e_i)=i$. Moreover, by \ref{item:labelexis3} the edges $e_0, \dots , e_{k-1}$ are in cyclic or anti-cyclic order on $\partial X$ depending on whether $L_{\X}(X)=\tt w$ or 
$L_{\X}(X)=\tt b$. So \ref{item:labelexis5} holds. 

\smallskip
Finally, to see that \ref{item:labelexis6} is true, let $(\tau_0, \tau_1, \tau_2)$ be a  flag in 
$\DD$. Then  $\tau_0=\{u\}$ for some $u\in {\bf V}$. The vertex $u$ is the initial point of the oriented edge  $\tau_1$. Let $v\in {\bf V}$ 
be  the terminal point of $\tau_1$, and define  $l=L_{\bf V}(u)$.

Depending on whether  the flag is positively- or negatively-oriented,  the vertex  $v$ follows $u$ in cyclic  or anti-cyclic 
order on   $\partial \tau_2$. So if the flag is positively-oriented, 
 then  
by property \ref{item:labelexis3} we have  $L_{\bf V}(v)=l+1$ if $L_\X(\tau_2)=\tt w$ 
and $L_{\bf V}(v)=l-1$ if $L_\X(\tau_2)=\tt b$. Property
\ref{item:labelexis4} 
now
implies that 
$L_{\E}(\tau_1)=l$ if $L_\X(\tau_2)=\tt w$ and 
$L_{\E}(\tau_1)=l-1$ if $L_\X(\tau_2)=\tt b$. 

So if $(\tau_0, \tau_1, \tau_2)$ is positively-oriented, then the cells in 
this flag carry the labels $l$, $l$, $\tt w$, or $l$, $l-1$, $\tt b$, respectively.

Similarly, if $(\tau_0, \tau_1, \tau_2)$ is negatively-oriented, then
we get the labels $l$, $l-1$, $\tt w$, or $l$, $l$, $\tt b$ for the
cells in the flag. Statement \ref{item:labelexis6} follows from this. 
\end{proof}

\section{Thurston maps from cell decompositions}
\label{sec:mapsfromdecomp} 
In Section~\ref{sec:tiles} we have seen how to obtain cell
decompositions from Thurston maps.  In this section we reverse
this procedure and ask when a pair $(\DD^1, \DD^0)$ of cell
decompositions gives rise to a Thurston map $f$ that is cellular
for $(\DD^1, \DD^0)$.  In general, one cannot expect $f$ to be
determined just by the cell decompositions alone, but one
needs additional information on how $f$ is supposed to map the
cells in $\DD^1$ to cells in $\DD^0$.  This is given by an
orientation-preserving labeling as discussed in the previous
section (see Definition~\ref{def:labeldecomp} and the discussion
following this definition).

We start with a lemma that allows us to recognize branched covering maps. 

\begin{lemma} \label{lem:constrmaps}
Let $\DD'$ and $\DD$ be two cell decompositions of $S^2$, and 
$f\: S^2\ra S^2$  be a cellular map for $(\DD', \DD)$ such that $f|X$ is  orientation-preserving for each tile $X$ in $\DD'$.

\begin{enumerate}

\item
\label{item:constr_map1} 
  Then $f$ is a branched covering map  on $S^2$. Each critical point
  of $f$ is a vertex of $\DD'$.   

\item
\label{item:constr_map2} 
  If in addition each vertex in  $\DD$ is also a vertex in $\DD'$, then every point in $\post(f)$ is a vertex of $\DD$. In particular, $f$ is postcritically-finite, and hence a Thurston map 
if $f$ is not a homeomorphism.   
\end{enumerate}
\end{lemma}

The assumption that $f|X$ is orientation-preserving means that $f$ preserves  the orientation of flags contained in $X$. 

\begin{proof} \ref{item:constr_map1}  
  We will show that for each point $p\in S^2$, there
  exist  topological disks $W'$ and $W=f(W')$ in $S^2$ with $p\in W'$ and $q=f(p)\in W$,  as well as  orientation-preserving homeomorphisms $\varphi\: W'\ra \D$
  and $\psi\: W\ra \D$  such that
  $\varphi(p)=0$, $\psi(q)=0$, and
$$ (\psi\circ f\circ \varphi^{-1})(z)=z^k$$ 
for all $z\in \D$, where $k\in \N$. 
 The desired relation between the points and maps can be represented by the 
commutative diagram 
 \begin{equation}\label{eq:digr} 
    \xymatrix{
      p\in W' \ar[r]^{f} \ar[d]^{\varphi} &  q\in W \ar[d]^{\psi}
      \\
      0\in \D  \ar[r]^{z\mapsto z^k} & 0\in \D\rlap{.}
    }
  \end{equation}

We will use the fact that 
if $f$ is an  orientation-preserving local 
homeomorphism near $p$, then we can take $k=1$ in \eqref{eq:digr} and  always find  suitable topological disks and   homeomorphisms.    

Let $p\in S^2$ be arbitrary. 
Since $S^2$ is the disjoint union of the interiors of the cells in $\DD'$, the point $p$ is contained in the interior of a tile or an edge in $\DD'$, or is a vertex of $\DD'$. Accordingly, we consider three cases. 

\smallskip
{\em Case 1:} There exists a tile $X'\in \DD'$ with $p\in \inte(X')$. Then $W'\coloneqq \inte(X')$ is an open neighborhood of $p$, and $f|W'$ is an orientation-preserving  homeomorphism of $W'=\inte(X')$ onto $W\coloneqq \inte(X)$, where $X=f(X')\in \DD$.  Hence $f$ is an orientation-preserving local homeomorphism near $p$.

\smallskip 
{\em Case 2:} There exists an edge  $e'\in \DD'$ with $p\in \inte(e')$. By Lem\-ma~\ref{lem:specprop}~\ref{item:prop_cell4}
 there exist distinct tiles $X',Y'\in \DD'$ such that $e'\sub \partial
 X'\cap \partial Y'$. Then the set $W'=\inte(X')\cup \inte(e')\cup \inte(Y')$
 is an open neighborhood of $p$. Since $f$ is cellular, $X=f(X')$ and
 $Y=f(Y')$ are tiles in $\DD$ and $e=f(e')$ is an edge in
 $\DD$. Moreover, $e\sub \partial X\cap \partial Y$. 

We  orient $e'$ so that $X'$ lies to the left and $Y'$ 
to the right of $e'$.  Since $f$ is orientation-preserving if restricted to tiles  in $\DD'$, the tile  $X$ lies to the left, and $Y$ to the right of the image $e$ of $e'$. In particular, $X\ne Y$, and so the sets
$\inte(X),\inte(e),\inte(Y)$ are pairwise disjoint, and their
union is open. Since $f$ is cellular and hence a homeomorphism if
restricted to cells (and interior of cells), it follows that 
the map
$f|W'$ is a homeomorphism of $W'$ onto the open set
$W=\inte(X)\cup \inte(e)\cup \inte(Y)$. 
Moreover, it is clear that $f|W'$ is orientation-preserving. 
Since $W'$ is open and contains $p$, the map $f$ is an
orientation-preserving local  homeomorphism near $p$.  

\smallskip
{\em Case 3:} The point $p$ is a vertex of $\DD'$. Since we now already know that in the complement of the vertex set of $\DD'$  our  map $f$ is an orientation-preserving local homeomorphism, one can deduce the desired local representation \eqref{eq:digr}  of $f$ near $p$ from a general fact (see Lemma~\ref{lem:z^d}). We will provide  a direct argument for this that  will also  give us additional insight how the cells in the cycle of $p$ are mapped (this is summarized in Remark~\ref{rem:dd'} after the proof).

As in the proof of Lem\-ma~\ref{lem:specprop}~\ref{item:prop_cell5}, we can choose 
tiles $X'_j\in \DD'$ and edges $e'_j\in \DD'$ for $j\in \N$ that  contain $p$ and satisfy 
$X'_j\ne X'_{j+1}$, $e'_j\ne e'_{j+1}$, and $e'_j\sub \partial X'_j\cap \partial X'_{j+1}$ for all $j\in \N$. There exists $d'\in \N$ 
such that $X'_{d'+1}=X'_1$, the tiles $X'_1, \dots, X'_{d'}$ as
well as the edges $e'_1, \dots, e'_{d'}$ are all distinct, and 
 $$W'=\{p\}\cup \inte(X'_1)\cup \inte(e'_1)\cup \inte(X'_2)\cup \dots \cup \inte(e'_{d'})$$
 is an open neighborhood of $p$ (this type of neighborhood is closely related to the concept of a flower; see Section~\ref{sec:flowers}).  Moreover, by the remark following 
 the proof of Lemma~\ref{lem:specprop}, we know that $X'_j=X'_{d'+j}$ and $e'_j=e'_{d'+j}$ for all $j\in \N$.

 Define $X_j=f(X'_j)$ and $e_j=f(e'_j)$ for $j\in \N$.  Since $f$
 is cellular for $(\DD',\DD)$, the set $X_j$ is a tile and $e_j$
 an edge in $\DD$. Note that
 $e_j\sub \partial X_j\cap \partial X_{j+1}$ for $j\in \N$. Since
 $X'_j$ and $X'_{j+1}$ are distinct tiles containing the edge
 $e'_{j}$ in their boundaries, it follows by an argument as in
 Case~2 above that $X_j\ne X_{j+1}$ for $j\in \N$. Similarly,
 $e'_{j}$ and $e'_{j+1}$ are distinct edges in $\DD'$ contained
 in $X'_{j+1}$, and $f|X_{j+1}' $ is a homeomorphism; so 
 $e_{j}\ne e_{j+1}$ for $j\in \N$.
 
 As in the proof of Lem\-ma~\ref{lem:specprop}~\ref{item:prop_cell5} we see that there exists a number $d\in \N$, $d\ge 2$, 
  such that $X_{d+1}=X_1$, and such that the tiles $X_1, \dots, X_d$ and the edges $e_1, \dots, e_d$ are all distinct. Moreover, 
 $$W=\{q\}\cup \inte(X_1)\cup \inte(e_1)\cup \inte(X_2)\cup \dots \cup \inte(e_d)$$ is an open neighborhood of $q=f(p)$, and $X_j=X_{d+j}$ and  $e_j=e_{d+j}$ for all $j\in \N$.
 
 The periodicity properties of the indexing of the tiles $X'_j$ and $X_j$ imply that $d\le d'$ and that $d$ is a divisor of $d'$. Hence there exists $k\in \N$ such that $d'=kd$. 
 
 We now claim that after  suitable coordinate changes  
 near $p$ and $q$, the map $f$ can be given the form $z\mapsto z^k$. 
 
 For $N\in \N$, $N\ge 2$, and $j\in \N$ define half-open  line segments 
 $$ R_j^N=\{re^{2\pi \iu j/N}: 0\le r<1\}\sub \D $$ and sectors
 $$\Sigma_j^N=\{re^{\iu t}:  2\pi (j-1)/N\le t\le  2\pi j/N
 \text{ and } 0\le r<1\}\sub \D. $$
  
 We then  construct a homeomorphism  $\psi\: W \ra \D$ with $\psi(q)=0$ as follows.   For each $j=1, \dots, d$ 
we first  map the  half-open arc  $\{q \}\cup \inte(e_j)$ homeomorphically to 
the half-open line segment $R_j^d$. 
 Then $q$ is mapped to $0$; so these maps are 
consistently defined for $q$. Since $X_j$ is a Jordan region, we can extend the homeomorphisms on $\{q\}\cup \inte(e_{j-1}) \sub \partial X_j$ and on $\{q\}\cup 
\inte(e_j)\sub \partial X_j$ to a homeomorphism of 
$$\{q\}\cup \inte(e_{j-1})\cup\inte(e_j)\cup \inte(X_j)$$ onto the sector  $\Sigma_j^d$ for each $j=2, \dots, d+1$. 
Since the sets 
$$\{q\}, \inte(e_1), \dots, \inte(e_{d}), \inte(X_2), \dots, \inte(X_{d+1})=\inte(X_1)$$ are pairwise disjoint and have $W$ as a union, 
these homeomorphisms paste together to a well-defined homeomorphism $\psi$ of $W$ onto $\D$. 
Note that 
$\psi(q)=0$ 
as well as
$\psi(X_j\cap W)=\Sigma_j^d$ for each $j=1, \dots, d$. 

We now define a map $\widetilde \varphi\:   \D\ra W'$ as follows. If $z\in \D$ is arbitrary, then $z\in \Sigma^{d'}_j$ for some $j=1, \dots, d'$.
Hence $z^k\in \Sigma^{d}_j$, and so $\psi^{-1}(z^k)\in X_j\cap W$. 
Since $f$ is a homeomorphism of $X'_j\cap W'$  onto $X_j\cap W$, it follows that $(f|X'_j)^{-1}(\psi^{-1}(z^k))$ is defined and lies in $X'_j\cap W'$.  
 
We set 
$$\widetilde \varphi(z)=(f|X'_j)^{-1}(\psi^{-1}(z^k)).$$
It is straightforward to verify that $\widetilde \varphi$ is well-defined and  a homeomorphism   of $\D$ onto $W'$ with
$\widetilde \varphi(0)=p$. It follows from the definition of $\widetilde \varphi$ that  
$(\psi\circ f\circ \widetilde \varphi)(z)=z^k$ for $z\in \D$. So if we set $\varphi=\widetilde \varphi^{-1}$, then $\varphi$ is a homeomorphism of $W$ onto $\D$ with $\varphi(p)=0$ and   we have the diagram \eqref{eq:digr}.   

Since  the tiles $X'_j$ and the edges $e_j'$ are indexed as in the proof of Lem\-ma~\ref{lem:specprop}~\ref{item:prop_cell5},  each  flag $(\{p\},e'_j, X'_{j+1})$ is positively-oriented
(see the remark after the proof of Lem\-ma~\ref{lem:specprop}). Since $f|X'_j$ is orientation-preserving, this  implies that the flag $(\{q\},e_j, X_{j+1})$ is also positively-oriented. Thus $\psi$ is orientation-pre\-ser\-ving, since $\psi$
maps  the positively-oriented flag $(\{q\},e_j, X_{j+1})$ in $S^2$  to the flag 
$(\{0\}, \overline R^d_j, \Sigma^d_{j+1})$ in $\CDach$ which is also positively-oriented. 
As follows from its definition, the map $\widetilde \varphi$ is then also orientation-preserving. Hence 
$\psi$ and $\varphi=\widetilde \varphi^{-1}$ are orientation-preserving homeomorphisms as desired.

 We have shown that in all cases the map $f$ has a local behavior as claimed. It follows that $f$ is a branched covering map. Moreover, we have seen that $f$  near each point is a local homeomorphism unless $p$ is a vertex of $\DD'$. It follows that each critical point of  $f$ is a vertex of $\DD'$. 

 \smallskip{} 
 \ref{item:constr_map2} 
 Suppose in addition that every
 vertex of $\DD$ is also a vertex of $\DD'$. Let $p$ be a critical
 point of $f$.  Then by \ref{item:constr_map1} the point $p$ is a vertex of $\DD'$. Since
 $f$ is cellular for $(\DD',\DD)$, the point $f(p)$ is a vertex of
 $\DD$.  Hence $f(p)$ is also a vertex of $\DD'$, and we can apply the
 argument again, to conclude that $f^2(p)$ is a vertex of $\DD$,
 etc. It follows that $\post(f)$ is a subset of the set of vertices of
 $\DD$. In particular, $\post(f)$ is finite, and so $f$ is
 postcritically-finite.
\end{proof}
 
 \begin{rem}\label{rem:dd'} Let  the map $f\:S^2 \ra S^2$ and the cell decompositions $\DD'$ and $\DD$ 
 be  as in the previous lemma, and let  $p$ be a vertex in $\DD'$. Then $q=f(p)$ is a vertex in $\DD$. If $d'$  and $d$ are   the lengths of the cycles of $p$ in $\DD'$ and  $q$ in $\DD$, respectively, then $d'=kd$, where $k=\deg_f(p)$. Moreover,  if $X'_j$ for $j\in \N$ are  the  tiles in the cycle of 
 $p$ in $\DD'$
labeled so that $X'_{d'+j}=X'_j$ for $j\in \N$, then  $X_j=f(X'_j)$ 
are the  tiles in the cycle  of $q$ in $\DD$. In addition,
 $X_{d+j}=f(X'_{d+j})=f(X'_j)=X_j$ for $j\in \N$. This was established  in Case~3 of the proof of Lemma~\ref{lem:constrmaps}.
 
 So on a more intuitive level, if we follow the tiles in the cycle of $p$ in a cyclic order modulo $d'$, then under the map we follow the tiles in the cycle of $q=f(p)$ also in a cyclic order modulo $d=d'/k$. Here each  tile $X_j=f(X'_j)$ in the image cycle has precisely $k=\deg_f(p)$ distinct  preimage tiles, 
 namely $X'_j, X'_{d+j}, \dots, X'_{(k-1)d+j}$.  A  similar statement is true for edges. \end{rem}

We can now prove the following  fact  which allows us the construction of many Thurston maps  (see Section~\ref{sec:examples-two-tile} for specific  examples).   

\begin{prop}\label{prop:thurstonex} Let $\DD^0$ and $\DD^1$ be  cell decompositions 
 of an oriented $2$-sphere $S^2$, and $L\: \DD^1\ra \DD^0$ be  an 
 orientation-preserving
  labeling. Suppose that every vertex of $\DD^0$ is also a vertex of $\DD^1$. 
 Then there exists a  branched covering map $f\:S^2\ra S^2$ that is cellular for $(\DD^1, \DD^0)$  and is compatible with  the given labeling 
 $L$. 
  The map $f$ is a homeomorphism or a Thurston map. In the latter case, 
  $f$ is unique up to Thurston equivalence, and $\post(f)$ is
  contained in the set of vertices of $\DD^0$. 
\end{prop}

 Later we will  consider  triples $(\DD^1, \DD^0, L)$ as in the previous proposition related to Thurston maps $f$ with invariant curves $\CC$. These triples  satisfy additional and  somewhat technical conditions and lead to the notion  of a {\em two-tile subdivision rule} (see Definition~\ref{def:subdivcomb}).

\begin{proof}[Proof of Proposition~\ref{prop:thurstonex}] As in the proof of Lemma~\ref{lem:isocellhomeo}~\ref{item:isocellhomeo2}, 
 a map $f$ as desired is obtained from  successive extensions to  the skeleta of the 
cell decomposition $\DD^1$. Indeed, let $L\:\DD^1\ra \DD^0$ be an orien\-tation-preserving labeling. If $v\in S^2$ is a $1$-vertex (i.e., a vertex in $\DD^1$), then $L(v)$ is a $0$-vertex (i.e., a vertex in $\DD^0$). 
Set $f(v)=L(v)$. This defines $f$ on the $0$-skeleton of $\DD^1$. 
To extend this to the $1$-skeleton   of $\DD^1$, let $e$ be an arbitrary $1$-edge. Then $e'=L(e)$ is a $0$-edge. Moreover, if $u$ and $v$ are the $1$-vertices that are the endpoints of $e$, then $u'=f(u)=L(u)$ and 
$v'=f(v)=L(v)$ are distinct $0$-vertices contained in  $e'$. Hence they are the endpoints of $e'$. So we can extend $f$ to $e$ by choosing a homeomorphism of $e$ onto $e'$ that agrees with $f$ on the endpoints of $e$. In this way we can continuously extend $f$ to the $1$-skeleton of $\DD^1$ so that $f|\tau$ is a homeomorphism of $\tau$ onto $L(\tau)$ whenever 
$\tau\in \DD^1$ is a cell of dimension $\le 1$. 

If $X$ is an arbitrary $1$-tile, then $\partial X$ is a subset of the $1$-skeleton of $\DD^1$  and hence $f$ is already defined on $\partial X$. Then   
$f|\partial X$ is a continuous  mapping of $\partial X$ into the boundary $\partial X'$ of the $0$-tile $X'=L(X)$.  The map $f|\partial X$ is  injective. Indeed, suppose that $u,v\in \partial X$ and $f(u)=f(v)$. Then there exist unique $1$-cells $\sigma, \tau\sub \partial X$ of dimension $\le 1$ such that $u\in \inte(\sigma)$ and $v\in \inte(\tau)$.  Then 
$$f(u)=f(v)\in \inte(f(\sigma))\cap \inte(f(\tau))=\inte(L(\sigma))\cap \inte (L(\tau))$$ 
and so the $1$-cells  $L(\sigma) $ and $L(\tau)$ must be  the same.  Since $L$ is a labeling and $\sigma, \tau \sub X\in \DD^1$,  it follows that $\sigma=\tau$. 
As the map  $f$ restricted to the  $1$-cell $\sigma=\tau$  is injective, we conclude   $u=v$ as desired.

So  $f|\partial X$ is a continuous and injective map  of $\partial X$ into $\partial X'$, and hence a homeomorphism between these sets. 
 We conclude that  $f$ can be extended to a homeomorphism of $X$ onto $X'$. These extensions on different $1$-tiles  paste together to a continuous map $f\:S^2\ra S^2$ that is cellular and compatible with  the given 
labeling. Moreover, $f|X$ is orientation-preserving for each $1$-tile $X$ as follows from the fact that the labeling is orientation-preserving. 
By Lemma~\ref{lem:constrmaps}~\ref{item:constr_map1} and \ref{item:constr_map2} the map $f$ is a postcritically-finite branched covering map. In particular, $f$ is a homeomorphism or a Thurston map. 
This shows that a map with the stated properties exists.

Suppose $f$ is a Thurston map. To show uniqueness up to Thurston equivalence, let  $g\:S^2\ra S^2$ be another continuous map that  is cellular for $(\DD^1, \DD^0)$ and compatible with $L$.  
We will prove that $g$ is a Thurston map that is equivalent to $f$. 

First note that  for each cell $\tau\in \DD^1$, the maps $f|\tau$ and $g|\tau$ are homeomorphisms of $\tau$ onto $L(\tau)\in \DD^0$.
Hence $\varphi_\tau\coloneqq (g|\tau)^{-1}\circ (f|\tau)$ is a homeomorphism 
of $\tau$ onto itself. The family $\varphi_\tau$, $\tau \in \DD^1$, of these homeomorphism is  obviously compatible under inclusions: if $\sigma,\tau \in \DD^1$ and 
$\sigma\sub \tau$, then $\varphi_\tau(p)=\varphi_\sigma(p)$ for all $p\in \sigma$.  

Using this, we can define a map $\varphi\:S^2\ra S^2$ as follows. For $p\in S^2$ pick $\tau\in \DD^1$ with $p\in \tau$. Then set $\varphi(p)\coloneqq\varphi_\tau(p)$. The compatibility properties  of the  homeomorphisms $\varphi_\tau$ imply that $\varphi$ is well-defined.  Indeed, suppose that $\tau, \tau'$ are cells in $\DD^1$ with $p\in \tau\cap \tau'$. There exists a unique cell $\sigma\in \DD^1$ with $p\in \inte(\sigma)$. It follows from Lemma~\ref{lem:celldecompint}~\ref{item:cell_decomp2} that 
$\sigma\sub \tau\cap \tau'$. 
Hence 
$$\varphi_\tau(p)=\varphi_\sigma(p)=\varphi_{\tau'}(p). $$

It is clear that $g\circ \varphi =f$. Moreover, $\varphi|\tau=\varphi_\tau$ is a homeomorphism of $\tau$ onto itself whenever $\tau\in \DD^1$. 
By 
Lemma~\ref{lem:isocellhomeo}~\ref{item:isocellhomeo1} 
and~\ref{item:isocellhomeo3} this implies   that $\varphi$ is a homeomorphism of $S^2$ onto itself that is isotopic to $\id_{S^2}$ rel.\ ${\bf V}^1$, where ${\bf V}^1$ is the set of $1$-vertices.

The set of postcritical points 
of $f$ is  contained in the set of $0$-vertices and hence in ${\bf V}^1$. So if we use the facts that $g\circ \varphi =f$ and  that $\varphi$ is isotopic to $\text{id}_{S^2}$ rel.~${\bf V}^1$, then Lemma~\ref{lem:T-eq_crit_post} (with  $h_1=\varphi$ and $h_0=\id_{S^2}$) implies  that 
$g$ is a Thurston map that is Thurston equivalent to $f$.
\end{proof}

\section{Flowers}\label{sec:flowers}
 
Throughout this section $f\:S^2\ra S^2$ is a given Thurston
map, 
and $\CC\sub S^2$ is a Jordan curve with $\post(f)\sub \CC$. We
consider the cell decompositions $\DD^n=\DD^n(f,\CC)$ and use the
related terminology and notation as discussed in
Section~\ref{sec:tiles}

The results in this section are based on the following concept.

\begin{definition}[$n$-Flowers]
  \label{def:flower}
  Let $n\in \N_0$, and  $p\in S^2$ be  an $n$-vertex.  Then the
  $n$-\defn{flower}\index{n9@$n$-!flower}\index{flower}\index{W n@$W^n(p)$} 
of $p$ is defined as
  \begin{equation*}
    W^n(p)\coloneqq   \bigcup\{\inte(c): c\in \DD^n, \ p\in c\}.   
    \end{equation*}
\end{definition}
So  the $n$-flower $W^n(p)$ of the  $n$-vertex $p$ is the union of the
interiors of all cells in the cycle of $p$ in $\DD^n$ (see Figure
\ref{fig:cycle} as well as 
Lemma~\ref{lem:specprop}~\ref{item:prop_cell5} and the discussion after this lemma).  

The main reason why we introduce flowers is the following.
Consider a simply 
connected region  $U\subset S^2$ not containing a postcritical
point of $f$ and branches $g_n$ of $f^{-n}$ defined on $U$. Then it may happen that  the number of $n$-tiles
intersecting $g_n(U)$ is unbounded as $n\to \infty$, even if the diameter of $U$ (with respect to some base metric on $S^2$) is
 small. For example, this happens when $f$ has a  periodic critical point $p$ (see Section~\ref{sec:periodic}), and $U$ spirals around one of the points  in the cycle generated by $p$. However, 
if $\diam (U)$ is sufficiently  small, then  $g_n(U)$ is always
contained in one $n$-flower as we shall see. Similar issues that
are resolved by the use of flowers are addressed in
Lemma~\ref{lem:difflevel} and Lemma~\ref{lem:tileflower}. 
 
 We first prove some basic properties of flowers. 

\begin{lemma}
  \label{lem:flowerprop}
  Let $n\in \N_0$, and $p\in S^2$ be an $n$-vertex.  As in
  Lemma~\ref{lem:specprop} let $e_1, \dots, e_d$ be the $n$-edges and
  $X_1, \dots, X_d$ be the $n$-tiles of the cycle of $p$, where $d\in
  \N$, $d\ge 2$, is the length of the cycle.


   \begin{enumerate}
    
  \item 
    \label{item:flower_prop1}
    Then $d=2\deg({f^n},p)$ and  the set $W^n(p)$ is homeomorphic to $\D$,
    i.e., it is an open, connected, and simply
    connected neighborhood of $p$. It contains no other $n$-vertex,
    and we have
    \begin{align}
      \label{eq:flowerrep}
      W^n(p)
      &=\{p\}\cup \bigcup_{i=1}^d\inte(X_i) \cup
      \bigcup_{i=1}^d\inte(e_i)\\
      \notag
      &=
      S^2\setminus \bigcup\{c\in \DD^n : c\in \DD^n, \ p\notin c\}.
    \end{align}


  \item
    \label{item:flower_prop2} 
    We have 
    \begin{equation*}
      \overline {W^n(p)}=X_1\cup \dots \cup X_d. 
    \end{equation*}
   Moreover,  the set $\partial W^n(p)$ is the union of all
    $n$-cells $c$ with $p\notin c$ and $c\sub \partial X_i$ for
    some $i\in \{1, \dots, d\}$.

    
      
  \item
    \label{item:flower_prop3}
    If $c$ is an arbitrary $n$-cell, then either $p\in c$ and $c\sub
    \overline{W^n(p)}$, or $c\sub S^2\setminus W^n(p)$.   
  \end{enumerate}
\end{lemma}

Note that by \ref{item:flower_prop1} each $n$-vertex $p$ is contained in precisely $d=2\deg(f^n,p)$ distinct $n$-edges and in precisely $d$ distinct $n$-tiles. 

\begin{proof} \ref{item:flower_prop1}  By Remark~\ref{rem:dd'}  the length $d$ of the cycle of the  vertex $p$ (in the cell decomposition $\DD^n$) is a multiple $d=k\widetilde d$ of the length $\widetilde d$ of the cycle of the image point $q=f^n(p)$ (in the cell decomposition $\DD^0$), where  $k$ is the degree of $f^n$ at $p$. Since $\widetilde d=2$, we have $d=2\deg({f^n},p)$ as claimed. 

The first equality in \eqref{eq:flowerrep} follows from Lemma~\ref{lem:specprop}~\ref{item:prop_cell5}. Based on this, the argument in Case~3 of the proof of Lemma~\ref{lem:constrmaps} shows that  there is  a homeomorphism of the set $W^n(p)$ onto 
 $\D$. Hence $W^n(p)$ is open, connected,  and  simply connected, and it follows 
 from the first equality in  
 \eqref{eq:flowerrep} that $W^n(p)$ contains no other $n$-vertex than $p$.

Let $M=S^2\setminus \bigcup\{c\in \DD^n : c\in \DD^n, \ p\notin c\}$.   If $x\in W^n(p)$, then $x$ is an interior point in one of the cells $\tau$ forming the cycle of $p$. So if $c$ is any $n$-cell with $x\in c$, then $\tau \sub c$ 
by  Lemma~\ref{lem:celldecompint}~\ref{item:cell_decomp2}. This implies $p\in c$, and so   $x\in M$ by definition of $M$. Hence $ W^n(p)\sub M$.

Conversely, if $x\in M$, let $\tau$ be an $n$-cell of smallest
dimension that contains $x$. Obviously, $x\in \inte(\tau)$. On
the other hand, the definition of $M$ implies that $p\in
\tau$.
Hence $\tau$ is a cell in the cycle of $p$, and so $x\in W^n(p)$.
We conclude $M\sub W^n(p)$, and so $M=W^n(p)$ as desired.

\smallskip{} 
\ref{item:flower_prop2} 
Equation \eqref{eq:flowerrep} implies
$\overline {W^n(p)}= X_1\cup \dots \cup X_n$.
  
Every point $x\in \partial W^n(p)$ is contained in one of the
sets $\partial X_i$. Note that $e_{i-1}, e_i\subset \partial
X_i$. 
Here we assume $n$-edges and $n$-tiles in the cycle of $p$ are
labeled as in Lemma~\ref{lem:specprop}~\ref{item:prop_cell5} and
we set $e_0=e_d$ for $i=1$ for convenience. 
Since $W^n(p)$ is open, the point $x$ is not
contained in $\{p\}\cup \inte(e_{i-1})\cup\inte(e_i)\sub W^n(p)$
 and hence is contained in an
$n$-cell $c$ in the boundary of $X_i$ distinct from $e_{i-1}$,
$e_i$, and $\{p\}$. Then
$p\notin c$, since $e_{i-1}$, $e_i$, and $\{p\}$ are the only
$n$-cells contained in $\partial X_i$ containing $p$. Thus 
$x$ is contained in an $n$-cell with the desired
properties.
  
Conversely, if $c$ is an $n$-cell with $p\notin c$ and
$c\sub \partial X_i$, then $c\sub S^2\setminus W^n(p)$ by
\eqref{eq:flowerrep}, and $c\sub X_i\sub \overline{W^n(p)}$.
Hence $c\sub \partial W^n(p)$.
 
\smallskip{}
\ref{item:flower_prop3} This follows from \ref{item:flower_prop1} and \eqref{eq:flowerrep}. 
  \end{proof}

Note that if we color tiles as in Lemma~\ref{lem:colortiles}, then the
colors of the tiles $X_1, \dots, X_d$ associated with  an  $n$-flower as
in the previous lemma will alternate.

\begin{lemma}\label{lem:mapflowers}
Let $k,n\in \N_0$. Then the following statements are true:

\begin{enumerate}

\item
\label{item:mapflowers1}
If $p\in S^2$ is an $(n+k)$-vertex, then $f^k$ maps the edges and tiles in the cycle of $p$  to the edges and tiles in the cycle of the $n$-vertex $q\coloneqq f^k(p)$  in  cyclic order in an $m$-to-$1$ fashion, where $m\coloneqq \deg(f^k,p)$.

Moreover, we have   $f^k(W^{n+k}(p))=W^n(q)$ and  
there exist orientation-preserving homeo\-mor\-phisms $\varphi\:   W^{n+k}(p)\ra \D$ and $\psi \: W^n(q)\ra \D$ with $\varphi(p)=0$ and $\psi(q)=0$ such that 
$$ (\psi\circ f^k\circ \varphi^{-1})(z)=z^m$$ 
for $z\in \D$. 

\item
\label{item:mapflowers2}
  If $q\in S^2$ is an $n$-vertex, then the  connected components of $f^{-k}(W^n(q))$ are the  $(n+k)$-flowers $W^{n+k}(p)$,  $p\in f^{-k}(q)$. 

\item
\label{item:mapflowers3}
  A connected set $K\sub S^2$ is contained in an $(n+k)$-flower if and only if $f^k(K)$ is contained in an $n$-flower.

  \item
\label{item:mapflowers4}
  The set of all $n$-flowers $W^n(p)$, $p\in \V^n$, is  an open cover of $S^2$.  
\end{enumerate} 

\end{lemma}

In the proof we will explain the precise meaning of the first statement in \ref{item:mapflowers1}. 

\begin{proof} \ref{item:mapflowers1} It is clear that $q=f^k(p)$ is an $n$-vertex. 
Let $e'_i$ and $X'_i$ for $i\in \N$ be   the $(n+k)$-edges and 
 $(n+k)$-tiles in the cycle of $p$, respectively.  Here we can choose the indexing as in the proof of Lemma~\ref{lem:specprop}~\ref{item:prop_cell5} so that 
 it has precise period $d'=2\deg(f^{n+k}, p)$, i.e., $e'_{d'+i}=e'_i$ and $X'_{d'+i}=X'_i$ for $i\in\N$
 and $d'$ is the smallest possible number here, because it is the length of the cycle of $p$   (see  Lemma~\ref{lem:flowerprop}~\ref{item:flower_prop1}). 
 
 Define $e_i=f^k(e'_i)$ and $X_i=f^k(X'_i)$ for $i\in \N$. Since  the map $f^k$ is cellular for $(\DD^{n+k}, \DD^n)$,
 it follows from Remark~\ref{rem:dd'} that 
$e_i$ and $X_i$ for $i\in \N$ are the $n$-edges  
and $n$-tiles  in the cycle of $q$. Here $e_{d+i}=e_i$ and $X_{d+i}=X_i$ for $i\in\N$ with 
$d=2\deg(f^{n}, q)$ and again $d$ is the smallest number with this property. In this sense, $f^k$ maps the edges and tiles in the cycle of $p$  to the edges and tiles in the cycle of $q$ in  cyclic order. 

The map $f^k$ between these cycles is $m$-to-$1$ with $m=\deg(f^k,p)$, because 
 each edge or tile in the cycle of $q$ has precisely
$$ m=d'/d=\deg(f^{n+k}, p)/\deg(f^n,q)=\deg(f^k,p)$$ distinct preimages in the cycle of $p$.

Note that we also have $f^k(\inte(e'_i))=\inte(e_i)$ and $f^k(\inte(X'_i))=\inte(X_i)$ for  $i\in \N$. 
This and  \eqref{eq:flowerrep} imply 
that $f^k(W^{n+k}(p))=W^n(q)$.

Finally, the  last statement in \ref{item:mapflowers1} follows from the considerations 
in Case~3 of the proof of Lemma~\ref{lem:constrmaps} (applied to the map $f^k$ and the 
cell decompositions $\DD^{n+k}$ and $\DD^n$).

\smallskip 
\ref{item:mapflowers2}
If $p\in f^{-k}(q)$, then $p$ is an $(n+k)$-vertex. By
\ref{item:mapflowers1} the $(n+k)$-flower $W^{n+k}(p)$ is an open
and connected subset of $f^{-k}(W^n(q))$. Suppose that
$x\in \partial W^{n+k}(p)$. 
Then by
Lemma~\ref{lem:flowerprop}~\ref{item:flower_prop2} there exist
an $(n+k)$-tile $X'$ and an $(n+k)$-cell 
$c'$ with $p\in X'$,
$p\notin c'$, and $x\in c'\sub \partial X'$. Then $X=f^k(X')$ is
an $n$-tile, $c=f^k(c')$ is an $n$-cell, $q\in X$, and
$f(x)\in c\sub \partial X$. Since $f^k|X'$ is a homeomorphism of
$X'$ onto $X$, we also have $q\notin c$.
Lemma~\ref{lem:flowerprop}~\ref{item:flower_prop2} implies that
$f^k(x)\in \partial W^n(q)$, and so $f^k(x)\notin W^n(q)$,
because flowers are open sets.

We conclude that $x\in S^2\setminus f^{-k}(W^n(q))$, and so
$\partial W^{n+k}(p)\sub S^2\setminus f^{-k}(W^n(q))$. 
It now
follows from Lemma~\ref{lem:conncomp} that $W^{n+k}(p)$ is a
connected component of 
\linebreak
$f^{-k}(W^n(q))$.

Conversely, suppose that $U$ is a connected component of the set
$f^{-k}(W^n(q))$.  Then $U$ is an open set and so it meets the
interior $\inte(X')$ of some $(n+k)$-tile $X'$. Then $X=f^k(X')$
is an $n$-tile that meets $W^n(q)$. Hence $q\in X$, and so there
exists an $(n+k)$-vertex $p\in X'$ with $f^k(p)=q$.

Then by the first part of the proof, the set $W^{n+k}(p)$ is a
connected component of $f^{-k}(W^n(q))$.  Since $W^{n+k}(p)$
contains the set $\inte(X')$ and so meets $U$, we must have
$W^{n+k}(p)=U.$





\smallskip 
\ref{item:mapflowers3}
Suppose  $K$ is contained in the $(n+k)$-flower $W^{n+k}(p)$. Then by
\ref{item:mapflowers1} the set $f^k(W^{n+k}(p))=W^n(f^k(p))$ is an
$n$-flower and it contains $f^k(K)$.  

Conversely, if $f^k(K)$ is contained in the $n$-flower $W^n(q)$, then
$K$ is a connected set in $f^{-k}(W^n(q))$. Hence $K$ lies in a
connected component of $f^{-k}(W^n(q))$, and hence in an
$(n+k)$-flower by \ref{item:mapflowers2}.

\smallskip 
 \ref{item:mapflowers4} We know by 
 Lemma~\ref{lem:flowerprop}~\ref{item:flower_prop1}
 that  flowers are open sets. If $x\in S^2$ is arbitrary, then there exists an $n$-cell $c$ such that 
 $x\in \inte(c)$. We can find an $n$-vertex
 $p$ such that $p\in c$. Then $x\in \inte(c)\sub W^n(p)$. So the $n$-flowers form indeed an open cover of $S^2$. 
\end{proof}

Similar to the definition of an  $n$-flower for an $n$-vertex,
one can also define an  
{\em edge flower}
for an $n$-edge. These sets provide ``canonical'' neighborhoods for $n$-vertices and $n$-edges
defined in terms of $n$-cells.

\begin{definition}[Edge flowers]
  \label{def:edgeflower}
  Let $n\in \N_0$, and $e$ be  an $n$-edge.  Then the  
 \defn{edge flower}\index{edge!flower}  
  of $e$ is defined as
  \begin{equation*}
    W^n(e)\coloneqq   \bigcup\{\inte(c)  : c\in \DD^n,\ c\cap e\ne \emptyset\}.   
    \end{equation*}
\end{definition}
 
We list some properties of edge flowers. They  correspond to similar properties of $n$-flowers as 
in Lemma~\ref{lem:flowerprop}. Note that in contrast to an
$n$-flower, an edge flower $W^n(e)$ will not be simply connected
in general (for example,  
if there is another 
$n$-edge $e'$ with the 
same endpoints as $e$ and $\#\post(f)\ge 3$).  

\begin{lemma}\label{lem:edgeflower}
Let $e$ be an $n$-edge whose endpoints are the  $n$-vertices $u$ and $v$. 

\begin{enumerate}

\item
  \label{item:edgeflower1}
  Then $W^n(e)$ is an open set containing $e$, and 
  \begin{equation} \label{eq:eflowerrep}
    W^n(e)=W^n(u)\cup W^n(v)=S^2\setminus  \bigcup\{c : c\in \DD^n,\
    c\cap e=\emptyset\}.
  \end{equation}
  
\item
\label{item:edgeflower2}
  We have $ \overline{W^n(e)}  = \bigcup \{X\in \X^n : X\cap e \neq
  \emptyset\}$. Moreover,  
  \begin{align*} 
    \partial {W^n(e)}= \bigcup \{c \in \DD^n : 
    {}&c\cap e=\emptyset \text{ and there exists} 
    \\ 
    &\text{$X\in \X^n$ with $X\cap e \ne \emptyset $ and $c\sub \partial X$}\},
  \end{align*} 
  where each  $n$-cell $c$ in the last union  either consists of one $n$-vertex or is an   $n$-edge. 

\item
  \label{item:edgeflower3}
  If $c$ is an arbitrary $n$-cell, then either $c\cap e\ne \emptyset$ and $c\sub \overline{W^n(e)}$, or $c\sub S^2\setminus W^n(e)$. 

\end{enumerate}
\end{lemma}

\begin{proof} \ref{item:edgeflower1} 
It follows from Lemma~\ref{lem:celldecompint}~\ref{item:cell_decomp1} that an  $n$-cell $c$  meets $e$ if and only if it contains one of the endpoints $u$ and $v$ of $e$.  Hence $W^n(e)=W^n(u)\cup W^n(v)$ by the definition of flowers. By Lemma~\ref{lem:flowerprop}~\ref{item:flower_prop1} this implies that $W^n(e)$ is open, and, since $e$ is an edge in the cycles of $u$ and $v$, we also have 
$$e=\{u\}\cup \inte(e)\cup\{v\}\sub  W^n(u)\cup W^n(v)=W^n(e).$$

Let $M= S^2\setminus  \bigcup\{c : c\in \DD^n,\ c\cap
e=\emptyset\}$. If an $n$-cell $c$ does not meet $e$, then it
contains neither $u$ nor $v$. 
Hence 
by \eqref{eq:flowerrep} we have 
\begin{equation*}
  S^2\setminus M\sub (S^2\setminus W^n(u))\cap 
 (S^2\setminus W^n(v))=S^2\setminus W^n(e), 
\end{equation*}
and so $ W^n(e)\sub M$.

 Conversely, let $x\in M$ be arbitrary, and $c$ be the unique
 $n$-cell $c$ such that $x\in \inte(c)$. Then $c\cap e\ne
 \emptyset $ and 
 therefore
 $u\in c$ or $v\in c$. It follows  that $x\in
 W^n(u)\cup W^n(v) =W^n(e)$. 
 We conclude that $M\sub W^n(e)$, and so $M=W^n(e)$ as claimed. 
 
 \smallskip 
 \ref{item:edgeflower2}
 By Lemma~\ref{lem:celldecompint}~\ref{item:cell_decomp1} an $n$-tile $X$ meets $e$ if and only if $X$ contains $u$ or $v$.
 Hence by \ref{item:edgeflower1} and Lemma~\ref{lem:flowerprop}~\ref{item:flower_prop2} we have
 $$ \overline{W^n(e)}=\overline{W^n(u)}\cup \overline{W^n(v)}=\bigcup \{X\in \X^n : X\cap e \neq \emptyset\}$$ as desired. 
 
 For the second claim suppose that 
  $c$ is an $n$-cell and  $X$ an $n$-tile with  
 $c\cap e=\emptyset$, $X\cap e\ne \emptyset$, and 
 $c\sub \partial X$. Then $c\sub S^2\setminus W^n(e)$ and $c$ must be an $n$-edge or consist of an $n$-vertex. 
 Moreover,  $c\sub X\sub \overline {W^n(e)}$. It follows that $c\sub \partial W^n(e)$.
 
 Conversely, let $x$ be a  point in 
 $\partial W^n(e)$. Then by \ref{item:edgeflower1} the point 
 $x$ is also a boundary point of $W^n(u)$ or $W^n(v)$, say 
 $x\in \partial W^n(u)$.

 By Lemma~\ref{lem:flowerprop}~\ref{item:flower_prop2} 
there exist an 
$n$-cell $c'$ and an $n$-tile $X$ with $x\in c'$, 
 $u\in X$, $u\notin c'$, and $c'\sub\partial X$.
 If $x$ is an $n$-vertex, we let $c=\{x\}$. Then $c$ is an $n$-cell and we have $c\cap e=\emptyset$, because $W^n(e)$ is an open neighborhood of $e$ and $c$ lies in $\partial W^n(e)\sub S^2\setminus W^n(e)$. 
 Moreover, $X\cap e\ne \emptyset $ and $c\sub c'\sub \partial
 X$. So $c$ is an  $n$-cell with the desired properties
 containing $x$.  
 
 If $x$ is not a vertex we put $c=c'$. Again if
 $c\cap e=c'\cap e=\emptyset$, then $c$ is an $n$-cell with the
 desired properties containing $x$.
 
 The other case, where $c\cap e\ne \emptyset$, leads to a
 contradiction. Indeed, then we have $v\in c$. Moreover, since
 $x$ is not a vertex, it follows that $x\in \inte(c)$. Note that
 $c$ then is necessarily an $n$-edge. It follows that
 $x\in \inte(c)\sub W^n(v)\sub W^n(e)$ which is impossible,
 because $x\in \partial W^n(e)\sub S^2\setminus W^n(e)$.
 
 \smallskip  
 \ref{item:edgeflower3}
 If $c$ is an $n$-cell and $c\cap e=\emptyset$, then $c\sub S^2\setminus W^n(e)$. If $c\cap e\ne \emptyset$, then 
 $c$ contains $u$ or $v$, and so  $c\sub \overline{W^n(u)}\cup \overline{W^n(v)}=\overline {W^n(e)}$. 
\end{proof}

\section{Joining opposite sides} 
\label{sec:opp}
In this section $f\:S^2\ra S^2$ will again be a Thurston map, and $\CC\sub S^2$ be a Jordan curve with
$\post(f)\sub \CC$. In addition, we assume that 
$\#\post(f)\ge 3$.   
We fix a base metric $d$ on $S^2$ that
induces the given topology and consider the cell decompositions
$\DD^n=\DD^n(f,\CC)$ as discussed in Section~\ref{sec:tiles}.

We will define a
constant $\delta_0>0$ such that any connected set of diameter
$< \delta_0$ (with respect to the base metric $d$) is contained
in a single $0$-flower (as introduced in
Section~\ref{sec:flowers}).  
However, there is a slight difference 
between the cases 
$\#\post(f)=3$ and $\#\post(f)\ge 4$.  In order to treat
these two cases simultaneously, the following definition is
useful.

\begin{definition}[Joining opposite sides]
  \label{def:connectop}
  \index{joining opposite sides|textbf}
  A set $K\subset S^2$ \defn{joins opposite sides}
  of $\CC$ 
      if $\#\post(f)\geq 4$ and $K$ meets  two disjoint $0$-edges, or    
    if $\#\post(f)=3$ and $K$ meets  all  three $0$-edges. 
 \end{definition}
 
 The case  $\#\post(f)=2$
is excluded here, and so the concept of ``joining opposite sides" (as well as the constant 
$\delta_0$ below) 
remains undefined for such Thurston maps $f$. 
 
We will mostly use Definition~\ref{def:connectop} for connected sets $K$ (when the phrase ``joining'' 
really makes sense), but it is convenient to allow arbitrary sets here.

We now  define\index{d0@$\delta_0$} 
 \begin{align}\label{defdelta}
 \delta_0
   =
   \delta_0(f,\CC)=\inf\{\diam(K)
   : 
   {}&K\sub S^2\text  { is a set}
  \\  \notag
   &\text{joining  opposite sides of } \CC\}.
 \end{align}

Then  $\delta_0>0$. Indeed,  if $\#\post(f)=4$, then   
$$\delta_0= \min\{\dist(e,e'):  e \text{ and } e' \text{ are disjoint }
  0\text{-edges}\}>0. $$
  
  If $\#\post(f)=3$ and we had $\delta_0=0$, then it would follow from a simple limiting argument that the three $0$-edges had a common point. This is absurd.

  \begin{lemma}\label{lem:floweropp}
  A connected set $K\sub S^2$ joins opposite sides of $\CC$ if and only if $K$ is not contained in a single  $0$-flower.
  \end{lemma}

  \begin{proof} If $K$ is contained in a $0$-flower $W^0(p)$, 
  where $p\in \CC$ is a $0$-vertex, then  
  $K$ meets at most two $0$-edges, namely the ones that have the common  endpoint $p$. So  $K$ does not join opposite sides of $\CC$. 
  
  Conversely, suppose $K$ does not   join opposite sides of $\CC$. We have to show that $K$ is contained in some $0$-flower. Note that $K$ cannot meet three distinct $0$-edges. 
    
  If $K$ does not meet any $0$-edge, then $K$ does not meet $\CC$ and is hence contained in the interior of one of the two $0$-tiles. This implies that $K$ is actually  contained in every $0$-flower. 
  
  If $K$ meets only one $0$-edge $e$, then 
$K$ is contained in the $0$-flowers $W^0(u)$ and $W^0(v)$, 
where $u$ and $v$ are the endpoints of $e$. 

If $K$ meets two edges, then these edges share a common endpoint $v\in {\bf V}^0=\post(f)$. This is always true if $\#\post(f)=3$ and follows from the fact that $K$ does not join opposite sides of $\CC$ if  $\#\post(f)\ge 4$. Moreover, $K$ cannot meet a third $0$-edge which  implies that  
$K\sub W^0(v)$. 
 \end{proof}

 By the previous lemma  every connected set $K\sub S^2$ satisfying $\diam(K)<\delta_0$ is contained in a $0$-flower.

\begin{lemma}
  \label{lem:preimsmall}
  Let $n\in \N_0$, and $\delta_0>0$ be as in \eqref{defdelta}. 
 
 \begin{enumerate}
 
 \item
   \label{item:preimsmall1}
   If $K\sub S^2$ is a connected set with  $\diam (K)< \delta_0$, then  every connected set $K'\sub  f^{-n}(K)$ is contained in some  $n$-flower.
 
 \item
   \label{item:preimsmall2}
   If $\gamma\colon [0,1]\to S^2$ is  a path such
  that $\diam (\gamma)< \delta_0$, then each  lift  $\widetilde{\gamma}$ of
  $\gamma$ by $f^n$ has an image that is contained in some  $n$-flower.

\end{enumerate}
 \end{lemma}
 
Here by definition a {\em lift} of $\ga$ by $f^n$ is any path $\widetilde{\gamma}\:[0,1]\ra S^2$ with ${\gamma}=f^n\circ \widetilde\ga$.  

\begin{proof}
  \ref{item:preimsmall1}
  The set $K$ is contained in some $0$-flower 
$W^0(p)$, $p\in {\bf V}^0$, by Lemma~\ref{lem:floweropp} and the definition of $\delta_0$. So if $K'$ is a connected subset  of $f^{-n}(K)$, then $K'$ is contained in a component of $f^{-n}(W^0(p))$,  and hence in an $n$-flower by 
 Lemma~\ref{lem:mapflowers}~\ref{item:mapflowers2}. 
 
 \smallskip{}
 \ref{item:preimsmall2}
 The reasoning is exactly the same as in \ref{item:preimsmall1}. The
(image of the) path  $\ga$ is contained in some $0$-flower; by
 Lemma~\ref{lem:mapflowers}~\ref{item:mapflowers2} this implies that
 any  lift  $\widetilde \ga$  of $\ga$ by $f^n$  is contained in an $n$-flower.  
\end{proof}

We will often have to estimate how many tiles are needed to connect
certain points. If we have a condition that is formulated ``at the top
level'', i.e., for connecting points in $\CC$, then the map $f^n$ can
be used to translate this to $n$-tiles.

\begin{lemma}\label{lem:maptotop} Let $n\in \N_0$, and $K\sub S^2$ be a connected set. 
If there exist two disjoint $n$-cells  $\sigma$ and $\tau$  with $K\cap \sigma\ne \emptyset$ and 
    $K\cap \tau \ne \emptyset$,  then $f^n(K)$ joins 
    opposite sides of $\CC$.
    \end{lemma}

\begin{proof}  It suffices to show that $K$ is not contained in any $n$-flower, because then $f^n(K)$ is not contained in any $0$-flower (Lemma~\ref{lem:mapflowers}~\ref{item:mapflowers3}) and so $f^n(K)$ joins opposite sides of $\CC$ (Lemma~\ref{lem:floweropp}).  We consider several cases. 

\smallskip
{\em Case~1:} One of the cells is an $n$-vertex,  
say $\sigma=\{v\}$, where $v\in {\bf V}^n$. Then $v\in K$;  so the only $n$-flower that $K$ could possibly be contained in is $W^n(v)$, because no other $n$-flower contains the $n$-vertex $v$. But since $\sigma$ and $\tau$ are disjoint, we have $v\notin \tau$, and so $\tau \sub S^2\setminus W^n(v)$. 
Hence $K\cap (S^2\setminus W^n(v))\ne \emptyset$, and so 
$W^n(v)$ does not contain $K$.  

\smallskip
{\em Case~2:} Suppose one of the cells is an $n$-edge,
say $\sigma=e\in \E^n$.  
Then $e$ has two endpoints $u,v\in {\bf V}^n$. The only $n$-flowers that meet $e$ are $W^n(u)$ and $W^n(v)$; so these $n$-flowers are the only ones that could possibly contain $K$. But the set $W^n(e)=W^n(u)\cup W^n(v)$ does not contain
$K$, because $K$ meets the set $\tau$ which lies in the complement of $W^n(e)$.   

\smallskip
{\em Case~3:} 
One of the cells is 
an $n$-tile, say $\sigma\in \X^n$.    Then $K$ meets $\partial X$. Since $\partial X$ consists of $n$-edges, the set $K$ meets an $n$-edge disjoint from $\tau$.  So we are reduced to Case~2. 
\end{proof} 

For $n\in \N_0$ we denote by  $D_n(f,\CC)$ the  minimal number of $n$-tiles required to form a connected set  joining   opposite
sides of $\CC$; more precisely,   
\index{d0 n@$D_n$|textbf}
\begin{align} 
  \label{def:dk}
  D_n(f,\CC) =\min\big\{&N\in \N: 
  {}\text{there exist } X_1,\dots, X_N\in \X^n 
  \text{ such that } 
  \\ \notag
  &K=\bigcup_{j=1}^N X_j 
  \text{ is connected and joins opposite sides of }
  \CC \big\}. 
\end{align} 
We simply write $D_n$ for $D_n(f,\CC)$ if $f$ and $\CC$ are clear from the context (as in this section). 

From Lemma~\ref{lem:maptotop} we can immediately  derive the following consequence. 

\begin{lemma} \label{lem:flowerbds}
Let $n,k\in \N_0$. Every set of $(n+k)$-tiles whose   union is connected 
and  meets  two disjoint $n$-cells  contains at least $D_k$ elements.
 \end{lemma}
 
 \begin{proof} Suppose $K$ is a union of $(n+k)$-tiles with the stated properties. Then the images of these tiles under $f^n$ are $k$-tiles and   $f^n(K)$ joins opposite sides of $\CC$ by Lemma~\ref{lem:maptotop}. 
 Hence there exist  at  least $D_k$  distinct $k$-tiles in the union forming $f^n(K)$ and hence at least $D_k$ distinct  $(n+k)$-tiles in $K$.  
 \end{proof} 

The following two lemmas give some  motivation why we introduced  
flowers. 
Namely, the number
of $(n-1)$-tiles or  the number of $(n+1)$-tiles
required to cover some $n$-tile $X$ may not be bounded by a
constant independent of $X$ and $n$. 
Similarly, in general there will  be no universal bound on the number of $n$-tiles defined with respect to a different Jordan curve $\widetilde \CC$ needed to cover $X$. 
Both issues  are resolved by
considering flowers instead of tiles. Note that in both lemmas we
allow $\#\post(f)=2$ for our given Thurston map $f$. 

\begin{lemma} \label{lem:difflevel} There exists $M\in \N$ with the following properties:

\begin{enumerate} 
\item
  \label{item:coverflower1}
  Each $n$-tile, $n\in \N$,  can be covered by  $M$ $(n-1)$-flowers.  

\item
  \label{item:coverflower2}
  Each $n$-tile, $n\in \N_0$,  can be covered by $M$ $(n+1)$-flowers.

\end{enumerate}
\end{lemma}

For easier  formulation of this lemma and  the subsequent proof, we assume for simplicity that a cover  by {\em at most} $M$ elements contains precisely $M$ elements. This can always be achieved by repetition of elements in the cover.

\begin{proof} 
  We first consider  the special  case when 
  $\#\post(f)=2$. Then there are exactly two $n$-vertices, 
  and hence
  exactly two $n$-flowers for each  $n\in \N_0$ (see
  Lemma~\ref{lem:postf-2}). These two $n$-flowers cover $S^2$
  (see Lemma~\ref{lem:mapflowers}~\ref{item:mapflowers4}). Thus
  both statements are true with $M=2$ in this case.
 
  Assume now that $\#\post(f)\geq 3$. 
  It suffices to consider the statements \ref{item:coverflower1}
  and \ref{item:coverflower2} separately and find a corresponding
  number $M$ for each of them.

 \smallskip
  \ref{item:coverflower1} 
  Let $\delta_0>0$ be as in \eqref{defdelta}. Then there exists
  $M\in \N$ such that each of the finitely many $1$-tiles $X$ is
  a union of $M$ connected sets $U\sub X$ with
  $\diam(U)<\delta_0$. If $Y$ is an arbitrary $n$-tile, $n\ge 1$,
  then $Z=f^{n-1}(Y)$ is a $1$-tile and $f^{n-1}|Y$ a
  homeomorphism of $Y$ onto $Z$.  Hence $Y$ is a union of $M$
  sets of the form $(f^{n-1}|Y)^{-1}(U)$, where $U\sub Z$ is
  connected and $\diam(U)<\delta_0$.  Each set
  $(f^{n-1}|Y)^{-1}(U)$ is connected and so by
  Lemma~\ref{lem:preimsmall}~\ref{item:preimsmall1} it lies in an
  $(n-1)$-flower.  Hence $Y$ can be covered by $M$
  $(n-1)$-flowers.

  \smallskip
  \ref{item:coverflower2} 
  There exists $M\in \N$ such that each of
  the two $0$-tiles $X$ can be covered by $M$ connected sets $U\sub X$
  with $\diam (f(U))<\delta_0$.  If $Y$ is an arbitrary $n$-tile, then
  $Z=f^n(Y)$ is a $0$-tile. By the same reasoning as above, the set
  $Y$ is a union of $M$ sets of the form $(f^{n}|Y)^{-1}(U)$, where
  $U\sub Z$ is connected and $\diam(f(U))<\delta_0$.

  Then $U'=(f^n|Y)^{-1}(U)$ is connected, and $f^{n+1}(U')=f(U)$ which
  implies $\diam(f^{n+1}(U'))<\delta_0$.  Hence by
  Lemma~\ref{lem:preimsmall}~\ref{item:preimsmall1} the set $U'$ is
  contained in some $(n+1)$-flower.  Since $M$ of the sets $U'$ cover
  $Y$, it follows that $Y$ can be covered by $M$
  $(n+1)$-flowers.
\end{proof}

\begin{lemma}\label{lem:tileflower} Let $\CC$ and $\widetilde
\CC$ be two Jordan curves in $S^2$ that both contain 
$\post(f)$. Then there exists a number $M$ such that 
each $n$-tile for $(f,\widetilde \CC)$, $n\in \N_0$, can be covered 
by $M$ $n$-flowers for $(f,\CC)$. 
\end{lemma}

\begin{proof} 
  The argument is very similar to the proof of
  Lemma~\ref{lem:difflevel}. Again the case $\#\post(f) =2$ is
  trivial;  so we may  assume $\#\post(f) \geq 3$.

  Let $\delta_0=\delta_0(f,\CC)>0$ be the number as defined in
  \eqref{defdelta}.  There exists a number $M$ such that each of the
  two $0$-tiles $X$ for $(f,\widetilde \CC)$ is a union of $M$
  connected sets $U\sub X$ with $\diam(U)<\delta_0$. If $Y$ is an
  arbitrary $n$-tile for $(f,\widetilde \CC)$, then $Z=f^{n}(Y)$ is a
  $0$-tile for $(f,\widetilde \CC)$ and $f^{n}|Y$ is a homeomorphism
  of $Y$ onto $Z$.  Hence $Y$ is a union of $M$ sets of the form
  $(f^{n}|Y)^{-1}(U)$, where $U\sub Z$ is connected and
  $\diam(U)<\delta_0$.  Each set $(f^{n}|Y)^{-1}(U)$ is connected and
  so by Lemma~\ref{lem:preimsmall}~\ref{item:preimsmall1} it lies in
  an $n$-flower for $(f,\CC)$.  Hence $Y$ can be covered by $M$ such
  $n$-flowers.
\end{proof}

\ifthenelse{\boolean{singlechapter}}{

%


\chapter{Expansion}
\index{expanding}
\index{Thurston map!expanding}
\label{cha:expansion}

 In this chapter we revisit the notion  of  expansion for 
 Thurston maps (see Definition~\ref{def:exp}) and study it in
 greater depth. We will establish basic properties of this
 concept.
 
In  Section~\ref{sec:defin-expans-revis} the main result  is
 Proposition~\ref{prop:expequivexp} which  gives several
 conditions that 
 are equivalent to our notion of expansion. 
 In particular,  one of these conditions (namely, 
 condition~\ref{item:expequivexp4} in 
 Proposition~\ref{prop:expequivexp}) can be formulated in terms of open covers 
 without  
 reference to a metric. This shows (as we remarked after Definition~\ref{def:exp})  that expansion is an 
entirely  topological property of a given Thurston map.   
 
 In Section~\ref{sec:further-results} we prove various other    results about
 expansion.  For example, in Lemma~\ref{lem:length_exp} we show that a Thurston map
 is expanding if it uniformly expands the length of paths with respect to an underlying  length 
 metric. This result was  already  used in our characterization
of rational expanding Thurston maps (see the proof of
Proposition~\ref{prop:rationalexpch}). It is not known if for every expanding Thurston map there is a
length metric with respect to which it is expanding.

In  Section~\ref{sec:Latttypeexp} we return to Latt\`{e}s-type
maps. We show that such a map is expanding if and only if each
eigenvalue of the  linear part $L_A$ of the affine  map  $A$  in 
Definition~\ref{def:Lattestype}  has absolute value $>1$ (see
Proposition~\ref{prop:expLattType}).

\section{Definition of expansion revisited}
\label{sec:defin-expans-revis}
 Let  $S^2$ be  a $2$-sphere. 
In the following, it is often convenient  to formulate some essentially
topological properties in metric terms. For this we 
 fix a base metric on  $S^2$ that induces the given 
topology. In this and the next section  notation for metric terms will refer to this base metric
unless otherwise indicated.

Let $f\: S^2\ra S^2$ be a Thurston map and 
  $\CC\subset S^2$ be   a Jordan curve with  $\post(f)\subset \CC$.  
 For $n\in \N_0$   we consider the cell decompositions
 $\DD^n=\DD^n(f,\CC)$ as given by Definition~\ref{def:DDn} 
with the 
corresponding set $\X^n=\X^n(f,\CC)$ of $n$-tiles. 
   Recall (from the beginning of Section~\ref{sec:expansion-1}) that 
   $\mesh(f,n, \CC)$ is defined as  the supremum of the diameters of the connected components of $f^{-n}(S^2\setminus \CC)=S^2\setminus f^{-n}(\CC)$. 
   We know that the   $n$-tiles for $(f,\CC)$ are precisely the closures of the connected
components of $S^2\setminus f^{-n}(\CC)$ (see 
Proposition~\ref{prop:celldecomp}~\ref{item:nedgesC}), and  
 so 
 \begin{equation*}
   \mesh(f,n, \CC)= \max_{X\in \X^n}\diam (X).   
 \end{equation*}
Thus a Thurston map $f$ is expanding (see Definition~\ref{def:exp}) if
and only if there is a Jordan curve $\CC\subset S^2$ with
$\post(f)\subset \CC$ such that
\index{mesh@$\mesh$}
\begin{equation}
  \label{eq:defexpXn}
  \max_{X\in \X^n}\diam (X)\to 0 \text{ as } n\to \infty,
\end{equation}
where the tiles are defined for  $(f,\CC)$. We record the
following immediate consequence.

\begin{lemma}
  \label{lem:no<3} 
  If $f\:S^2\ra S^2$ is an expanding Thurston map, then
  $\#\post(f)\ge 3$.
\end{lemma} 

\begin{proof} 
  By Corollary~\ref{cor:post012} we know that   $\#\post(f)
  \ge 2$.
  
  If $\#\post(f)=2$, then there exist two distinct points
  $p,q\in S^2$ with $\post(f)=\{p,q\}$. Let $\CC\sub S^2$ be an
  arbitrary Jordan curve with $\post(f)\sub \CC$, and consider
  the set $\X^n$ of $n$-tiles for $(f,\CC)$. Then every $n$-tile
  $X$ contains $p$ and $q$ (see Lemma~\ref{lem:postf-2}). Thus
  \begin{equation*}
      \max_{X\in \X^n}\diam (X) \geq d(p,q) > 0
  \end{equation*}
 for all $n\in \N_0$,  where $d$ denotes the fixed base metric on
  $S^2$. This means that $f$  cannot be  expanding. 
\end{proof} 

Due to this lemma, we can always assume that $\#\post(f) \geq 3$  when we consider expanding Thurston maps $f$.

Let us now convince ourselves that condition \eqref{eq:defexpXn}
is independent of the choice of the curve $\CC$.  

%

\begin{lemma}
  \label{lem:exp_ind_C} 
  \index{Thurston map!expanding}
  \index{expanding}
  Let $f\:S^2 \ra S^2$ be a Thurston map and
  $\CC,\widetilde{\CC}\subset S^2$ be Jordan curves with
  $\post(f)\subset \CC,\widetilde{\CC}$. Then 
    \begin{equation*}
      \lim_{n\to \infty} \mesh(f,n, \CC)=0
      \text{ if and only if } \lim_{n\to \infty}
    \mesh(f,n, \widetilde{\CC})=0. 
    \end{equation*}
%
%
\end{lemma}

\begin{proof}
  Let $\CC, \widetilde{\CC}\subset S^2$ be as in the statement of the
  lemma. Assume that $\lim_{n\to \infty} \mesh(f,n, \CC)=0$. This means
  that $f$ is expanding. Then 
  $$ \max_{X\in \X^n} \diam(X)=\mesh(f,n,\CC) \to 0$$ 
  as $n\to \infty$, where $\X^n$ is the set of  $n$-tiles for
  $(f,\CC)$.  Lemma~\ref{lem:flowerprop}~\ref{item:flower_prop2}
  implies that    
  \begin{equation}\label{diamflower}
    \diam(W^n(p))\le 2 \max_{X\in \X^n} \diam(X)
  \end{equation} for each $n$-flower $W^n(p)$ for $(f,\CC)$. 

  Now we consider  tiles for $(f,\widetilde \CC)$. By
  Lemma~\ref{lem:tileflower}  there exists a number $M\in \N$
  such that each $n$-tile for $(f, \widetilde
  \CC)$ can be  covered by $M$  $n$-flowers for $(f,\CC)$. 
 If a connected set is covered by a finite union of connected sets, then its diameter is bounded by the sum of the diameters of the sets in the union. 
  Combining this with \eqref{diamflower}, we conclude 
 that  
 \begin{align*} 
   \mesh(f,n,\widetilde \CC)&=
 \max \{\diam(\widetilde X): \widetilde X \text{ is an $n$-tile
                              for $(f, \widetilde \CC)$}\} 
\\ &\le M \max_{p\in {\bf V}^n} \diam(W^n(p)) 
\\ 
 &\le 
 2M\max_{X\in \X^n} \diam(X)
\\
&= 2M\mesh(f,n,\CC).
\end{align*}
Here $\mathbf{V}^n$ denotes  the set of $n$-vertices for $(f,\CC)$. 
The last  inequality implies that
  $\lim_{n\to\infty} \mesh(f,n,\widetilde \CC)= 0$  as desired.

  The other implication is obtained by reversing the roles of $\CC$ and
  $\widetilde\CC$. 
\end{proof}
  
The lemma shows that  a Thurston map $f\colon S^2 \to S^2$ is
expanding\index{expanding} if and only if 
$\mesh(f, n, \CC) \to 0$ as $n\to \infty$ for all Jordan curves 
 $\CC\subset S^2$ 
 with $\post(f)\subset
\CC$.   In particular,  expansion is
a property of the map $f$ alone and  independent of 
the choice of the  Jordan curve $\CC$.

Lemma~\ref{lem:tileflower}, which was used in the previous proof,
admits an improvement for expanding Thurston maps.

\begin{lemma} 
  \label{lem:tileflowerimprov}
  Let $f\: S^2\ra S^2$ be an expanding Thurston map. Suppose that
  $\CC$ and $\widetilde \CC$ are two Jordan curves in $S^2$ that
  both contain $\post(f)$. Then there exists a number $M\in \N$
  with the following property: if $n,k\in \N_0$, then every
  $(n+k)$-tile for $(f,\widetilde \CC)$ can be covered by $M$
  $n$-flowers for $(f,\CC)$. 
\end{lemma}

\begin{proof} 
  The argument is a small variation of the one that we used to
  establish Lemma~\ref{lem:tileflower}. Note that
  $\#\post(f)\geq 3$, since $f$ is expanding (see
  Lemma~\ref{lem:no<3}). 

  Let $\delta_0=\delta_0(f,\CC)>0$ be the number as defined in
  \eqref{defdelta}. Since $f$ is expanding, there exists a
  number $M\in \N$ such that {\em each} tile $X$ for
  $(f,\widetilde \CC)$ is a union of $M$ connected sets $U\sub X$
  with $\diam(U)<\delta_0$ (in the proof of
  Lemma~\ref{lem:tileflower} we could guarantee this only for the
  two $0$-tiles for $(f,\widetilde \CC)$). Indeed, since $f$ is
  expanding this is trivially true for all tiles $X$ of
  sufficiently high levels, because then
  $\diam(X)<\delta_0$. There are only finitely many tiles $X$ for
  $(f,\widetilde \CC)$ with $\diam(X)\ge\delta_0$. The existence
  of a suitable constant $M$ easily follows.
  
Now let  $n,k\in \N_0$ and suppose  $Y$ is   an
  arbitrary $(n+k)$-tile for $(f,\widetilde \CC)$. Then $Z=f^{n}(Y)$ is a
  $k$-tile for $(f,\widetilde \CC)$ and $f^{n}|Y$ is a homeomorphism
  of $Y$ onto $Z$.  Hence $Y$ is a union of $M$ sets of the form
  $(f^{n}|Y)^{-1}(U)$, where $U\sub Z$ is connected and
  $\diam(U)<\delta_0$.  Each set $(f^{n}|Y)^{-1}(U)$ is connected and
  so by Lemma~\ref{lem:preimsmall}~\ref{item:preimsmall1} it lies in
  an $n$-flower for $(f,\CC)$.  Hence $Y$ can be covered by $M$ such
  $n$-flowers.
\end{proof}

Our definition of expansion is somewhat {\em ad hoc}, but it has
the advantage that it relates to the geometry of tiles. As we
will see, equivalent and maybe more conceptual descriptions can
be given in terms of the behavior of open covers of $S^2$ under
pull-backs by the iterates of the map.  This shows that expansion is a
topological property of the map. Our definition was based on a
metric concept (namely the mesh size), but this was just for
convenience.

We start with some definitions. Let $\mathcal{U}$ be an open
cover of $S^2$. We define
$\mesh(\mathcal{U})$\index{mesh@$\mesh$} 
to be the supremum of
all diameters of connected components of sets in
$\mathcal{U}$. If $g\: S^2\ra S^2$ is a continuous map, then the
{\em pull-back of $\mathcal{U}$ by $g$} is defined as
$$ g^{-1}(\mathcal{U})=\{V : V \text{ is a connected component of }
g^{-1}(U), \text{ where } U \in \mathcal{U}\}.$$ 
Obviously, $g^{-1}(\mathcal{U})$ is also an open cover of
$S^2$. Similarly,  we denote by $g^{-n}(\mathcal{U})$ the pull-back of
$\mathcal{U}$ by $g^n$. 

\begin{prop}
  \label{prop:expequivexp}
  \index{expanding} 
  \index{Thurston map!expanding}
  Let $f\: S^2\ra S^2$ be a Thurston map. 
  Then the following conditions are equivalent:
 
\begin{enumerate}

\item 
  \label{item:expequivexp1}
  The map $f$ is expanding. 

\item 
  \label{item:expequivexp2}
  \index{d0@$\delta_0$} 
  There exists $\delta_0>0$ with the following property: if  $\mathcal{U}$ is a cover of $S^2$ by open and connected sets that satisfies  $\mesh(\mathcal {U})<\delta_0$, then
$$\lim_{n\to\infty} 
\mesh (f^{-n}(\mathcal {U}))=0.$$ 

\item
  \label{item:exp_cover}
  There exists an open cover $\mathcal{U}$ of $S^2$ with  
  $$\lim_{n\to\infty} 
  \mesh (f^{-n}(\mathcal {U}))=0. $$ 
  
\item 
  \label{item:expequivexp4}
  There exists an open cover $\mathcal{U}$ of $S^2$ with  
  the following property: for every open cover  $\mathcal{V}$  of $S^2$
there exists $N\in \N$ such that $f^{-n}(\mathcal{U})$ is finer than
$\mathcal{V}$ for every  $n\in \N$ with  $n>N$, i.e.,  
for every set $U'\in f^{-n}(\mathcal {U})$ there exists a
set  $V\in \mathcal{V}$ such that $U'\sub V$.  
\end{enumerate} 
\end{prop} 

Condition \ref{item:exp_cover} is the notion of expansion as defined
by Ha\"\i ssinsky-Pilgrim (see \cite[Section~2.2]{HP}). So our notion of
expansion agrees with the one in \cite{HP}.  Condition
\ref{item:expequivexp4} is essentially a reformulation of
\ref{item:exp_cover} in purely topological terms without reference to
the base metric on $S^2$ (which enters in the definition of the mesh
of an open cover). One can reformulate \ref{item:expequivexp2} in a
similar spirit. We will see in the proof below that the constant
$\delta_0$ in \ref{item:expequivexp2} can be chosen to be the
number from \eqref{defdelta}. 
If there exists a Jordan curve $\CC\sub S^2$ with
$\post(f)\sub \CC$ and $f(\CC)\sub \CC$, then expansion of the map $f$
can be characterized in yet another way (see
Lemma~\ref{lem:charexpint}).

\begin{proof} We will show
  \ref{item:expequivexp1} $\Rightarrow$~\ref{item:expequivexp2}
  $\Rightarrow$~\ref{item:exp_cover} $\Rightarrow$~\ref{item:expequivexp1} 
  and
  \ref{item:exp_cover} $\Rightarrow$~\ref{item:expequivexp4} $\Rightarrow$~\ref{item:exp_cover}.

\smallskip 
 \ref{item:expequivexp1} $\Rightarrow$~\ref{item:expequivexp2}
 Suppose 
 that 
 $f$ is expanding. Pick  a Jordan curve $\CC\sub S^2$
 with $\post(f)\sub \CC$, and let $\delta_0>0$ be as in
 \eqref{defdelta} (note that $\#\post(f)\ge 3$ by Lemma~\ref{lem:no<3}). Suppose $\mathcal{U}$ is a cover of $S^2$ by open and connected sets that satisfies  $\mesh(\mathcal {U})<\delta_0$. If $U\in \mathcal {U}$, 
then $U$ is connected and $\diam(U)<\delta_0$. So if $V$ is an arbitrary connected component of $f^{-n}(U)$, then by 
Lemma~\ref{lem:preimsmall}~\ref{item:preimsmall1} 
the set $V$ is contained in an $n$-flower for $(f,\CC)$. 
Hence 
$$ \diam(V)\le 2 \mesh(f,n,\CC), $$ which implies 
$$\mesh(f^{-n}(\mathcal {U}) )\le2 \mesh(f,n,\CC). $$ Since $f$
is an expanding Thurston map, we have $\mesh(f,n,\CC )\to 0$,
and
hence
$\mesh(f^{-n}(\mathcal {U})) \to 0$ as $n\to \infty$. 

\smallskip 
\ref{item:expequivexp2} $\Rightarrow$~\ref{item:exp_cover} This is obvious.

\smallskip 
\ref{item:exp_cover} $\Rightarrow$~\ref{item:expequivexp1} Suppose $\mathcal {U}$ is an open cover of $S^2$ as in \ref{item:exp_cover}.    
Pick  a Jordan curve $\CC\sub S^2$ with $\post(f)\sub \CC$, and let
$\delta>0$ be a
{\em Lebesgue number} 
\index{Lebesgue!number} 
for the cover 
$\mathcal {U}$, i.e., every set $K\sub S^2$ with $\diam(K)<\delta$ is contained in  
a set $U\in \mathcal {U}$. We can find a number $M\in \N$ such that each of the two $0$-tiles for $(f,\CC)$ can be written as a union of $M$  connected sets  $V$ with $\diam(V)<\delta$. Then each such set $V$ is contained in a set $U\in  \mathcal {U}$. 

Now if  $X$ is an arbitrary $n$-tile for $(f, \CC)$, then $Y=f^{n}(X)$ is a $0$-tile for $(f,\CC)$ and $f^{n}|X$ is  a homeomorphism of $X$ onto $Y$. 
Hence $X$  is a union of  $M$  connected sets  of the form $(f^{n}|Y)^{-1}(V)$, where $V\sub Y$ is connected and lies in a set $U\in \mathcal{U}$.  Then   
$(f^{n}|X)^{-1}(V)$ lies in a component of $f^{-n}(U)$, and so 
$$\diam ((f^{n}|X)^{-1}(V))\le \mesh(f^{-n}(\mathcal {U}) ). $$
This implies 
$$\diam (X)\le M \mesh(f^{-n}(\mathcal {U}) ).$$
Hence $$\mesh(f,n,\CC)\le M  \mesh(f^{-n}(\mathcal {U})). $$ Since $
\mesh(f^{-n}(\mathcal {U}))\to 0$, we also have $\mesh(f,n,\CC)\to 0$
as $n\to \infty$. 
It follows that  $f$ is expanding. 

\smallskip 
\ref{item:exp_cover} $\Rightarrow$~\ref{item:expequivexp4} Suppose $\mathcal {U}$ is an open cover of $S^2$ as in \ref{item:exp_cover}, and  $\mathcal{V}$ is  an arbitrary open cover of $S^2$.  Let $\delta>0$ be a  Lebesgue number for the cover 
$\mathcal {V}$, i.e., every set $K\sub S^2$ with $\diam(K)<\delta$ is contained in  
a set $V\in \mathcal {V}$. By \ref{item:exp_cover} we can find $N\in \N$ such that 
$\mesh(f^{-n}(\mathcal {U}))<\delta$ for $n>N$. If $n>N$ and $U'$ is a set in 
$f^{-n}(\mathcal {U})$, then $\diam(U')<\delta$ by definition of
$\mesh(f^{-n}(\mathcal {U}))$. Hence there exists $V\in \mathcal{V}$
such that $U'\sub V$.   

\smallskip 
\ref{item:expequivexp4}
$\Rightarrow$~\ref{item:exp_cover} Suppose $\mathcal{U}$ is an open
cover of $S^2$ as in \ref{item:expequivexp4}. Then $\mathcal{U}$ also satisfies condition
\ref{item:exp_cover}; indeed, let $\eps>0$ be arbitrary, and let
$\mathcal{V}$ be the 
open cover of $S^2$ consisting of all open balls of radius $\eps/2$.
Then $\diam(V)\le \eps$ for all $V\in \mathcal{V}$. Moreover, by
\ref{item:expequivexp4} 
there exists $N\in \N$ such that for $n>N$ every set in
$f^{-n}(\mathcal {U})$ is contained in a set in $\mathcal{V}$. In
particular, $\mesh(f^{-n}(\mathcal {U})) \le \eps $ for $n>N$.  This
shows that $\mathcal{U}$ satisfies condition \ref{item:exp_cover}.
\end{proof}

\section{Further results on expansion}
\label{sec:further-results}

In this section we collect various other useful results  related to
expansion. 

\begin{lemma}
  \label{lem:Thiterates} 
  \index{expanding} 
  \index{Thurston map!expanding}
  Let $f\: S^2\ra S^2$ be  a Thurston map, $n\in \N$, and
  $F=f^n$. Then $F$ is a Thurston map with $\post(F)=\post(f)$. 
   The map $f$ is expanding if and only if $F$ is expanding.
\end{lemma}

\begin{proof} Since $f$ is a Thurston map, the map  $F$ is a branched  covering map  on $S^2$ 
with $\post(F)=\post(f)$ (see Section~\ref{sec:defin-thurst-maps}) and
$\deg(F)=\deg(f)^n\ge 2$. Hence $F$ is also a Thurston map.

Fix a Jordan curve $\CC\sub S^2$ with $\post(f)=\post(F)\sub \CC$. 
It follows from the definitions that 
$$\mesh(F,k,\CC)=\mesh(f,nk,\CC) $$ for all $k\in \N_0$.
If $f$ is expanding, then by Lemma~\ref{lem:exp_ind_C} we have 
$$\lim_{k\to \infty} \mesh(f,k,\CC) =0 $$  which implies that 
$$\mesh(F,k,\CC)=\mesh(f,nk,\CC)\to 0$$ as $k\to \infty$. Hence $F$ is expanding.

Conversely, suppose that $F$ is expanding. Then we know that 
\begin{equation} \label{Fexp}
 \lim_{k\to \infty} \mesh(F,k,\CC)= \lim_{k\to \infty} \mesh(f,nk,\CC)=0. 
 \end{equation} 

Let the constant $M\ge 1$   be as in 
Lemma~\ref{lem:difflevel} for the map $f$ and the Jordan curve $\CC$. By an argument similar to the proof of 
Lemma~\ref{lem:exp_ind_C} one can show  that 
$$\mesh(f, l+1, \CC)\le 2M\mesh(f, l, \CC) $$
for all $l\in \N_0$. 
This implies  
$$ \mesh(f, l, \CC) \le (2M)^n \mesh(f, n \lfloor l/ n\rfloor, \CC)$$ for all $l\in \N_0$ and so by \eqref{Fexp} we have 
$ \mesh(f,l, \CC)\to 0$ as $l\to \infty$. This shows  that $f$ is expanding. 
\end{proof}

A map $f\colon S^2\to S^2$ is called \emph{eventually
  onto}\index{eventually onto|textbf}, if for each non-empty  open set $U\subset S^2$
there is an iterate $f^n$ such that $f^n(U)=S^2$. 

\begin{lemma}
  \label{lem:event_onto}
  Let $f\colon S^2\to S^2$ be an expanding Thurston map. Then $f$ is
  eventually onto.
\end{lemma}

As we will see (Example~\ref{ex:non-expanding-lattes}), there are
Thurston maps that are eventually onto, but not expanding.

\begin{proof}
  Let $f\colon S^2\to S^2$ be an expanding Thurston map. Pick a Jordan
  curve $\CC\sub S^2$ with $\post(f)\sub \CC$ as in
  Definition~\ref{def:exp}. We consider tiles for $(f,\CC)$. 
  As before, we denote the  black and white  $0$-tiles for $(f,\CC)$
  by  $\XOb$ and $\XOw$, respectively.



  Let $U\subset S^2$ be an arbitrary  
  non-empty open set, and   $B(a,\epsilon)$ with $a\in U$ and $\eps>0$  be an open ball contained in
  $U$. Since $f$ is expanding, there is a number  $n\in \N$ such that
  $\mesh(f,n, \CC)< \epsilon/4$. Then  each $n$-tile has
  diameter $<\epsilon/4$. Let $X$ be an $n$-tile containing the center $a\in U$ of $B(a,\eps)$, and  $Y$ an
  $n$-tile that shares an $n$-edge with $X$. Then  $X\cup Y \subset
  B(a,\epsilon)\subset U$, and so  $f^n(U) \supset
  f^n(X\cup Y) = \XOw\cup \XOb= S^2$. The claim follows. 
\end{proof}

A metric $d$ on a space $S$ is called a
\emph{length metric}\index{length!metric}\index{metric!length} 
or \emph{path metric}\index{path!metric}\index{metric!path} 
(see Section~\ref{sec:metrspterm})
 if
for any two points $x,y\in S$ we have  $d(x,y)= \inf_{\ga}  \length(\gamma)$,
where the infimum is taken over all paths $\gamma$ in $ S$ joining  $x$
to $y$. Using this concept,  one can formulate  a simple criterion when a Thurston map is expanding.

\begin{lemma}
  \label{lem:length_exp}
  Let   $d$ be   a length metric on $S^2$ that induces the given 
  topology on $S^2$, and  let $f\colon S^2\to S^2$ be a Thurston map. 
  If  $f$  uniformly expands the $d$-length
  of paths, i.e., if there is a number $\rho>1$ such that for every 
  path  $\gamma$ in $S^2$ we have 
  \begin{equation*}
    \length_d(f\circ \gamma) \geq \rho \length_d(\gamma),  
  \end{equation*}
  then $f$ is expanding. 
\end{lemma}

\begin{proof}
  Let $d$ be a length metric on $S^2$ such that the Thurston map
  $f\colon S^2\to S^2$ expands the $d$-length of  paths as in the
  statement of the lemma. In the following, all metric notions refer to this metric $d$.
  To prove that $f$ is expanding,  we will show
  that condition \ref{item:exp_cover} in 
  Proposition~\ref{prop:expequivexp} is satisfied for a suitable cover $\mathcal{U}$ of $S^2$.

We pick a Jordan curve $\CC\sub S^2$ with $\post(f)\sub \CC$, and consider cells  for 
$(f,\CC)$. Then  the corresponding $0$-flowers $W^0(p)$, $p\in \post(f)$, form a cover of $S^2$ (see Lemma~\ref{lem:mapflowers}~\ref{item:mapflowers4}). 
In order to obtain a cover $\mathcal{U}$ as in Proposition~\ref{prop:expequivexp}~\ref{item:exp_cover}, we want to shrink each $0$-flower $W^0(p)$ slightly to a new set 
$U$ so that we have good  control for the length of paths joining points in $U$ to $p$ inside  $W^0(p)$. 
Note that since $d$ is a length metric and $W^0(p)$ is open and connected, every point in $W^0(p)$ can be joined to $p$ by a path in $W^0(p)$ of finite length, but 
in general there will be no uniform upper bound for the length of these paths. 

In order to obtain such a bound, let $r>0$ and define $W_r^0(p)$ for $p\in \post(f)$  to be the set of all points $u\in W^0(p)$ such that $u$ and $p$  can be joined by a path $\ga$ in $W^0(p)$ with $\length(\ga)<r$. Then $W_r^0(p)$ is  open and $W_r^0(p)\sub W^0(p)$. 

\smallskip 
{\em Claim.} There exists $r>0$ such that each point in $S^2$ is contained in one of the sets $W_r^0(p)$, $p\in \post(f)$. 

\smallskip 
To prove  this, let $u\in S^2$ be arbitrary. Since the  $0$-flowers cover $S^2$,  there exists 
$p\in \post(f)$ such that $u\in W^0(p)$. Then  we can find  a path $\ga$ in $W^0(p)$ joining $u$ and $p$ with 
 $r_u\coloneqq \length(\ga)<\infty$. 
 
 We can choose $\delta_u>0$ such that $B_u\coloneqq B(u, \delta_u)\sub W^0(p)$. Since $d$ is a length metric,  every point $v$ in  $B_u$ can be joined with $u$ by a path in $B_u$ of length $<\delta_u$. If  we concatenate such a path with $\ga$, then we obtain a path  that has length $<r_u+\delta_u$ and stays inside $W^0(p)$.  In particular, we have uniform control for the length of such paths for all points in $B_u$. Since finitely many of the balls $B_u$, $u\in S^2$, cover $S^2$, the claim follows.

Now pick $r>0$ as in the claim, and consider the open cover $\mathcal{U}$ of $S^2$ given by the sets $W^0_r(p)$, $p\in \post(f)$.  Let $n\in \N_0$ and $p\in \post(f)$ be arbitrary, and consider a component $V$ of $f^{-n}(W^0_r(p))$. 
Then $V$ is  contained in a component of  $f^{-n}(W^0(p))$, and so there exists an $n$-flower $W^n(q)$ such that $V\sub W^n(q)$ (see Lemma~\ref{lem:mapflowers}~\ref{item:mapflowers2}). Here $q\in S^2$ is an $n$-vertex. Then $f^n(q)$ is a $0$-vertex contained in 
$W^0(p)$ which implies that $f^n(q)=p$. 

Let $v\in V$ be arbitrary, and $u\coloneqq f^n(v)\in W^0_r(p)$. Then
there exists a path $\ga$ in $W_r^0(p)\sub W^0(p)$ with
$\length(\ga)<r$ that joins $u$ and $p$.  By
Lemma~\ref{lem:liftsofpathsbranched} there exists a lift
$\alpha$ of $\ga$ by $f^n$ that starts at $v$. Then
$f^n\circ \alpha=\ga$ and $\alpha\sub V\sub W^n(q)$. One
endpoint of $\alpha$ is $v$, while the other endpoint of
$\alpha$ is a preimage of $p$ under $f^n$ and hence an
$n$-vertex. Since $q$ is the only $n$-vertex in $V\sub W^n(q)$,
it follows that $\alpha$ joins $v$ and $q$.

By using the fact that $f$ expands the $d$-length of  paths by
the factor $\rho$, we see  that
$$\length (\alpha)\leq \frac{1}{\rho^n}\length (f^n\circ \alpha)=\frac{1}{\rho^n}\length(\ga)<
\frac {r}{\rho^n}. $$ 
So every point in $V$ can be joined to $q$ by a path of length $<
r/\rho^n$. 
This implies that $\diam(V)\leq 2r/\rho^n$, 
and it follows that $\mesh(f^{-n}(\mathcal{U}))\leq  2r/\rho^n$. 
Since $\rho>1$, we conclude that $\mesh(f^{-n}(\mathcal{U}))\to 0$ as $n\to \infty$. By Proposition~\ref{prop:expequivexp}, the map $f$ is expanding.   \end{proof}

A Thurston map $f\colon S^2\to S^2$ is called a \defn{Thurston
  polynomial}\index{Thurston!polynomial} if there exists a point in $S^2$, denoted by 
$\infty$,   that is completely invariant, i.e.,
$f^{-1}(\infty)=\{\infty\}$.  

\begin{lemma}
  \label{lem:poly_not_exp}
No Thurston polynomial $f$ is  expanding.
\end{lemma}

\begin{proof}
  Let $f$ be a Thurston polynomial. We can choose a point    $\infty\in S^2$  that is completely invariant. Then  $ \deg_f(\infty)
  =\deg(f)\geq 2$ by \eqref{eq:sum_deg}, and so  $\infty$ is a critical 
  point of $f$. Since $\infty$ is a fixed point as well, it follows that
  $\infty\in \post(f)$.

 Let    $\CC\subset S^2$ be an arbitrary Jordan curve  with $\post(f)\subset \CC$. We consider tiles for $(f,\CC)$.  Each  $n$-tile $X^n$ is
  mapped by $f^n$ homeomorphically to a $0$-tile $X^0$ (see
  Proposition~\ref{prop:celldecomp}~\ref{item:fkcellular}). Since
  $$\infty\in \post(f)\sub \CC = \partial X^0\sub X^0,$$ it follows that $X^n$ contains a preimage of
  $\infty$ by $f^n$. Since $\infty$ is completely invariant, the only
  such preimage is $\infty$ itself, and so $\infty\in X^n$. Therefore,  each $n$-tile contains
  $\infty$. 

We pick a point   $p\in S^2$ distinct from
  $\infty$. Since for each $n\in \N$ the set of all  $n$-tiles forms a cover of $S^2$,
  there exists an $n$-tile $X^n$ containing $p$.  Then  
  $$ \mesh(f,n, \CC)\ge \diam (X^n) \ge  d(p,\infty)>0, $$ where 
  $d$ denotes the  base metric on $S^2$. It follows that  
 $\mesh(f,n, \CC)\not \to  0$ as $n\to \infty$.  This means  that $f$ is not
  expanding.   
\end{proof}

Let $f\colon S^2\to S^2$ be a Thurston map.  
A 
\emph{Levy cycle}\index{Levy cycle} 
for $f$ is a multicurve $\Gamma=
\{\gamma_1, \dots, 
\gamma_{n}\}$ (see
Definition~\ref{def:multicurve}~\ref{item:def_multicurve}) with the following property: 
for each $j=1, \dots, n$ the set $f^{-1}(\gamma_{j+1})$ contains a component
  $\widetilde{\gamma}_{j}$ that is
  isotopic to $\gamma_{j}$ rel.\ $\post(f)$ such that the map  
  $f|\widetilde{\gamma}_{j} \colon\widetilde{\gamma}_{j}\to \gamma_{j+1}$   is a
  homeomorphism (here we set $\ga_{n+1}=\ga_1$).

Since $f|\widetilde{\gamma}_{j} \colon\widetilde{\gamma}_{j}\to \gamma_{j+1}$ is  a covering map, the last condition is equivalent to the requirement that the (unsigned) degree of this map is equal to $1$. 

\begin{lemma}
  \label{lem:istop_gamma_pre}
  Let $f\colon S^2 \to S^2$ be a Thurston map and suppose that 
  $\gamma$ and  $\sigma$ are  Jordan curves in $S^2\setminus
  \post(f)$ that are isotopic rel.\ $\post(f)$. Let $\gamma_1,  \dots
  , \gamma_k$ with $k\in \N$ be the components of $f^{-1}(\gamma)$. Then $f^{-1}(\sigma)$  has also $k$ components. Moreover,   we can  label them 
 as $\sigma_1, \dots , \sigma_k$
  such that for $j=1, \dots, k$, 
  \begin{enumerate}
  
  \item the curves $\gamma_j$ and $\sigma_j$ are isotopic rel.\
    $\post(f)$, 
    
  \item the degrees of $f|\gamma_j  \colon \gamma_j \to \gamma$ and  $f| \sigma_j  \colon
    \sigma_j \to \sigma$ agree.
  \end{enumerate}
\end{lemma}

\begin{proof}  To see this, we   lift a suitable  isotopy  by
  $f$. 
So let
$H\colon S^2
  \times [0,1] \to S^2 $ be an isotopy rel.\ $\post(f)$ that deforms $\gamma$ to
  $\sigma$, i.e., $H_0= \id_{S^2}$ and $H_1(\gamma) =
  \sigma$. 
Then by Proposition~\ref{prop:isotoplift}  (that we will establish later) 
there is an isotopy $\widetilde{H}\colon S^2
  \times [0,1] \to S^2$ rel.\ $\post(f)$ with $\widetilde{H}_0=
  \id_{S^2}$ such that 
  \begin{equation}
    \label{eq:liftHbyf}
      (H_t\circ f)(p) = (f\circ\widetilde{H}_t)(p) 
  \end{equation}
  for all $p\in S^2$, $t\in [0,1]$.

 Define 
  $\sigma_j \coloneqq  \widetilde{H}_1(\gamma_j)$. Then  by
  definition
  $\gamma_j$ and $\sigma_j$ are isotopic rel.\ $\post(f)$ for $j=1, \dots, n$. Moreover, \eqref{eq:liftHbyf} implies  
  that $\sigma_1, \dots, \sigma_k$ are the components of
  $f^{-1}(\sigma)$ and that the degrees of $f|\gamma_j \colon \gamma_j \to \gamma$
  and  $f| \sigma_j\colon \sigma_j \to \sigma$ agree. \end{proof}

Now suppose that  $f$ is  a Thurston map that has a Levy cycle $\Gamma=\{\gamma_1,
\dots , \gamma_{n}\}$. We  consider the iterate $F= f^n$ and define
$\gamma\coloneqq \gamma_1$.  If we use the previous lemma repeatedly, 
then we see  that there is a component
$\widetilde{\gamma}$ of $F^{-1}(\gamma)$ that is isotopic to
$\gamma$ rel.\ $\post(f)$ such that  the degree of $F|\widetilde{\gamma} \colon
\widetilde{\gamma} \to \gamma$ is $1$. The existence of such an iterate $F=f^n$
and such  a (non-peripheral) Jordan curve $\gamma\subset S^2
\setminus \post(f)$  is in
fact equivalent to the existence of a Levy cycle, but we will not
prove this here. 

By using a lifting argument as in the proof of the previous
lemma (based on Proposition~\ref{prop:isotoplift}),  one can
easily show that Levy cycles 
persist under Thurston equivalence. So
if the Thurston maps $f\colon S^2\to S^2$ and $g\colon S^2 \to S^2$ are equivalent, then $f$ has a Levy cycle if and only if
$g$ has a Levy cycle.

If a Levy cycle $\Gamma$ is an {\em invariant} multicurve, then
it is clearly a Thurston obstruction, since the spectral radius
of the corresponding Thurston matrix $A(f,\Gamma)$ is $\ge 1$.
 If the Levy cycle $\Gamma$ is not invariant, then it is
not hard to show that there is an invariant multicurve
$\Gamma' \supset \Gamma$ (see \cite[Lemma~2.2]{Le92}).
 Then the spectral radius of
$A(f,\Gamma')$ is $\geq 1$, and so every Levy cycle is contained
in a Thurston obstruction. This means that  if a Thurston map has a Levy cycle
(and a hyperbolic orbifold), then it cannot be equivalent to a
rational map according to Thurston's theorem 
(Theorem~\ref{thm:Thurston}).

Our next lemma shows that a   Levy cycle is also an obstruction for  a Thurston map to be 
 expanding. 

\begin{lemma}
  \label{lem:levy_not_exp}
  Let $f\colon S^2 \to S^2$ be a Thurston map that has a Levy
  cycle. Then $f$ is not expanding. 
\end{lemma}

Of course, Levy cycles are not the only obstructions for expansion of a Thurston map. For example, by Lemma~\ref{lem:poly_not_exp} no holomorphic Thurston polynomial is expanding, but, as follows from  Thurston's theorem, such a polynomial cannot have Levy cycles either.

\begin{proof}
  Assume $f$ has a Levy cycle. Then there exists an iterate $F=f^n$
  and a non-peripheral Jordan curve $\gamma_1\subset S^2 \setminus
  \post(f)$ such that a component $\gamma_2$ of $F^{-1}(\gamma_1)$ is
  isotopic to $\gamma_1$ rel.\ $\post(f)$ and the degree of $F| \gamma_2\colon
  \gamma_2\to \gamma_1$ is equal to $1$. From Lemma~\ref{lem:istop_gamma_pre}
  it follows by induction that there is a sequence $\{\gamma_k\}_{k\in
    \N}$ of Jordan curves in $S^2\setminus \post(f)$ all of which are
  isotopic to $\gamma_1$ rel.\ $\post(f)$ such that $\gamma_{k+1}$ is a
  component of $F^{-1}(\gamma_k)$ and the degree of $F| \gamma_{k+1} \colon
  \gamma_{k+1} \to \gamma_{k}$ is equal to  $1$ for all $k\in \N$. Hence
  $\gamma_k$ is a component 
  of $F^{-k}(\gamma_1)$ and the degree of $F^k| \gamma_{k+1}\colon \gamma_{k+1} \to
  \gamma_1$ is equal to $1$. Moreover, each curve $\gamma_k$ is non-peripheral. 

  Since $\gamma_1$ is non-peripheral, we have 
  $\#\post(f) \geq 4$ (each component of
  $S^2\setminus \gamma_1$ contains at least two postcritical
  points).  We fix a Jordan curve $\CC\sub S^2$ with
  $ \post(f)=\post(F)\sub \CC$, and consider tiles and flowers
  for $(F,\CC)$.  Let $\delta_0$ be as in \eqref{defdelta} for
  the map $F$. Then we can decompose $\gamma_1$ into finitely
  many arcs $\alpha_1, \dots , \alpha_l$ such that
  $\diam(\alpha_j)< \delta_0$ for $j=1,\dots, l$. By
  Lemma~\ref{lem:preimsmall}~\ref{item:preimsmall1} every
  connected subset of $F^{-k}(\alpha_j)$ is contained in a
  $k$-flower $W^k_j$ for $j=1, \dots , l$. It follows that for
  each $k\in \N$ the curve $\gamma_k$ is contained in the union
  of $l$ $k$-flowers $W^k_1, \dots, W^k_l$.

 To reach a contradiction,  assume now that $f$, and hence also $F$, is expanding. Then
 we have   \begin{equation*}
    \diam(\gamma_k) \leq \sum_{j=1}^l \diam(W^k_j) \leq 2l \max_{X\in \X^k} \diam (X)
    \to 0
  \end{equation*}
  as $k\to \infty$. 
  
 On the other hand, we can find a finite open cover $\mathcal{U}$ of $S^2$ consisting of 
simply
  connected regions $U$  each of which contains at most
  one postcritical point of $f$ (for example, the $0$-flowers form such a cover). By what we have just seen,  we can find $k\in \N$ such
  that $\diam (\gamma_k)$ is smaller than the Lebesgue number of
  $\mathcal{U}$. Then  $\gamma_k$ is contained in a set
  $U\in \mathcal{U}$. Since $U$ is simply connected and contains at most one postcritical point of $f$,  the curve  $\gamma_k$ is peripheral. This is a contradiction showing  that
  $f$ is not expanding.   
\end{proof}

\begin{ex}
  \label{ex:no_per_crit_not_exp}
  We now present an example of a Thurston map $f$ with a Levy cycle. The map
$f$ will have no periodic critical points, and so it provides an example of a
Thurston map without periodic critical points that is not expanding,
contrasting Proposition~\ref{prop:rationalexpch}. Since Levy cycles
persist under Thurston equivalence,  $f$ is also not equivalent to any
expanding Thurston map. 
%

 For the construction we start with a topological sphere $S^2$
 that is a pillow (see Section~\ref{sec:expratThmaps}). Similar to Section~\ref{sec:Lattes}, the
 pillow is obtained by gluing two unit squares together
 along their boundaries. As before, we color one side (i.e., one
 square) of the pillow white, and the other black.

  The black side is divided horizontally into two rectangles, one of which is colored white and the other colored  black. 
  The white side of the pillow is subdivided into four quadrilaterals,
  two white and two black ones as indicated on  the left in 
  Figure~\ref{fig:Levy_cycle}. Here we have cut the pillow along three
  sides so that we obtain a rectangle as shown in the picture. The
  symbols in the picture indicate which sides have to be identified to
  recover the pillow. 
  
  \begin{figure}
  \centering
  \begin{overpic}
    [width=12cm, tics=5,
    ]
    {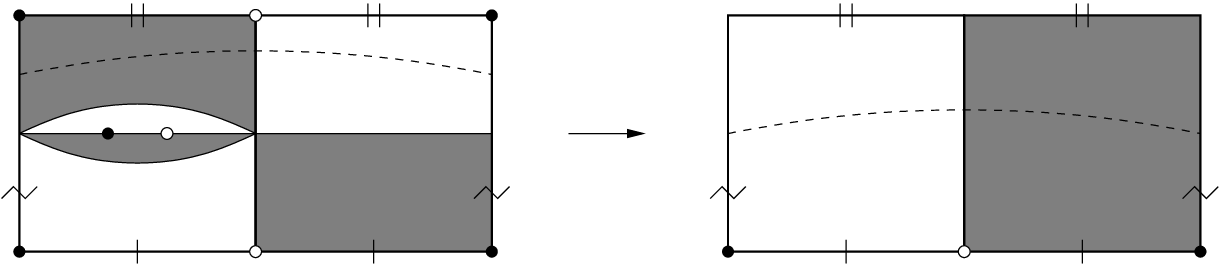}
    \put(49,13){${\scriptstyle f}$}
    \put(99,11){${\scriptstyle \gamma}$}
    \put(41,15){${\scriptstyle \widetilde{\gamma}}$}
    \put(-1,10){${\scriptstyle c_1}$}
    \put(21.5,11.5){${\scriptstyle c_2}$}
    \put(0.5,21.5){${\scriptstyle p_1}$}
    \put(39.5,21.5){${\scriptstyle p_1}$}
    \put(0.5,-1){${\scriptstyle q_1}$}
    \put(39.5,-1){${\scriptstyle q_1}$}
    \put(20,21.5){${\scriptstyle p_2}$}
    \put(20,-1){${\scriptstyle q_2}$}
    \put(58.5,21.5){${\scriptstyle p_1}$}
    \put(97,21.5){${\scriptstyle p_1}$}
    \put(58.5,-1){${\scriptstyle q_1}$}
    \put(97.5,-1){${\scriptstyle q_1}$}
    \put(78,21.5){${\scriptstyle p_2}$}
    \put(78,-1){${\scriptstyle q_2}$}
  \end{overpic}
  \caption{A map with a Levy cycle.}
  \label{fig:Levy_cycle}
\end{figure}
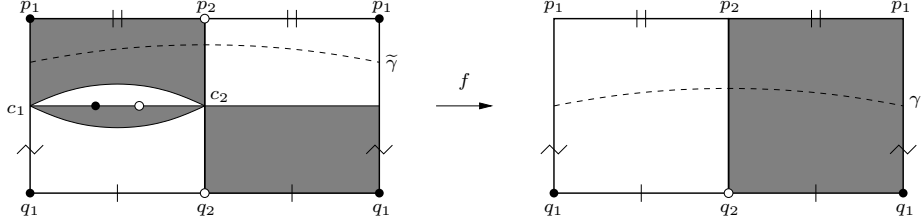

Now $f$ is constructed by mapping  each white quadrilateral homeomorphically 
  to the white face, and each black quadrilateral to the black face of
  the pillow. Here $f$ maps vertices to vertices. 
  In Figure~\ref{fig:Levy_cycle} we have marked two
  vertices of each quadrilateral  (on the left), as well as two
  vertices of the pillow (on the right) by a black or white dot to indicate the correspondence of vertices under 
  $f$. Finally,  we require that $f$ agrees on sides shared by
  two quadrilaterals. The map $f$ thus defined is indeed a Thurston map (because it realizes a two-tile subdivision rule; see
  Chapter~\ref{cha:subdivisions} and in particular Proposition~\ref{prop:rulemapex}). The
  postcritical points correspond to the vertices of the pillow. The
  map $f$ has two critical points $c_1$ and $c_2$ and  the following ramification portrait:
    \begin{equation*}
    \xymatrix @R=1pt{
      c_1 \ar[r]^{3:1} & p_1  \ar[r] & q_1 \ar@(r,u)[]
    }
    \quad\qquad
    \xymatrix @R=1pt{
      c_2 \ar[r]^{3:1} & p_2  \ar[r] & q_2\rlap{.}  \ar@(r,u)[]
    }
  \end{equation*}
  Thus $f$ has no periodic critical points,   its 
  signature is $(3,3,3,3)$, and  $f$ has a hyperbolic orbifold. 

  We consider the Jordan curve $\gamma\subset S^2 \setminus \post(f)$ as indicated  on
  the right in  Figure~\ref{fig:Levy_cycle}. One of the components
  $\widetilde{\gamma}$ of $f^{-1}(\gamma)$ is shown on the left  (the
  other components of $f^{-1}(\gamma)$ are not shown). Clearly, 
  $\widetilde{\gamma}$ is isotopic rel.\ $\post(f)$ to
  $\gamma$. Furthermore,  the degree of $f\colon \widetilde{\gamma}\to
  \gamma$ is equal to $1$, and so  $\Gamma=\{\gamma\}$ is a Levy cycle. 

 Lemma~\ref{lem:levy_not_exp} implies that $f$ is not
  expanding, and,  by our earlier discussion, no Thurston map 
  equivalent to $f$ is  expanding. 
\end{ex}

\section{Latt\`es-type maps and expansion} 
\label{sec:Latttypeexp}

We know (see
Theorem~\ref{thm:Lattesstruc}~\ref{item:Lattessrucii} and (i')) that every Latt\`es map is expanding. This is not always true for Latt\`es-type maps, but  it is easy to decide when this is the case.  The relevant condition is based on the following   definition. 

Let $L\: \R^2 \ra \R^2$ be a linear map. We call it {\em expanding} if $|\lambda|>1$ for each of the two (possibly complex) eigenvalues $\lambda$ of $L$. 

\begin{prop} 
  \label{prop:expLattType} 
  Let $f\:S^2\ra S^2$ be a Latt\`es-type map and $L=L_A$ be the
  linear part of an affine map $A$ as in
  Definition~\ref{def:Lattestype}.  Then $f$ is expanding (as a
  Thurston map) if and only if $L$ is expanding (as a linear
  map). 
\end{prop}

For the proof we require two lemmas. 

\begin{lemma} \label{lem:explinmap} Suppose $L\: \R^2 \ra \R^2$ is an expanding linear map. Then there exist constants 
  $\rho>1$ and $n_0\in \N$ such that 
 \begin{equation}
  \label{eq:Lmetric_exp}
 |L^n(v)|\ge \rho^n |v| 
\end{equation} for all $v\in \R^2$ and all $n\in \N$ with $n\ge n_0$.
 \end{lemma}
 
 Here $|v|=\sqrt{x^2+y^2}$ denotes the usual Euclidean norm of $v=(x,y)\in \R^2$. 

\begin{proof} 
The eigenvalues of $L$ are the two (possibly identical) roots  $\lambda_1, \lambda_2\in \C$ of the characteristic polynomial  $P(\lambda)=\det(L-\lambda \id_{\R^2})$
of $L$.  
We may assume $|\lambda_1|\le |\lambda_2|$. Since $L$ is expanding we have $|\lambda_1|>1$. 
The polynomial  $P$ has real coefficients, and so  $\lambda_2=\overline{ \lambda_1}$ if $\lambda_1$ is not real.  

There exists a basis of $\R^2$ consisting of two linearly independent vectors $v_1,v_2\in \R^2$ such that 
the linear map $L$ has a matrix representation with respect to this basis of one of the following forms:
$$M_1=|\lambda_1| \left(\begin{array}{cc} \cos \theta  & -\sin \theta \\
\sin \theta  & \cos \theta
\end{array} \right), \text{where $\theta\in \R$}, $$
$$  M_2=\left(\begin{array}{cc} \lambda_1  & 0 \\
0 & \lambda_2
\end{array} \right), \text{ or } M_3=\left(\begin{array}{cc} \lambda_1  & 1 \\
0 & \lambda_1
\end{array} \right).  
$$

We can find an inner product on $\R^2$ such that $v_1$ and $v_2$ form an orthonormal basis of $\R^2$ with respect to this inner product. If $\Vert \cdot \Vert$ is  the norm induced by this product, then 
$$ \Vert av_1+bv_2\Vert =\sqrt {a^2+b^2}$$ 
for $a,b\in \R$. If $L$ has a matrix representation given by  $M_1$ or $M_2$, it is clear that 
$ \Vert L(v) \Vert \ge |\lambda_1|\cdot  \Vert v\Vert $ 
and so 
 \begin{equation}
  \Vert L^n(v) \Vert \ge |\lambda_1|^n \cdot \Vert v\Vert 
 \end{equation} 
 for all $v\in \R^2$ and $n\in \N$. If $L$ is represented by the matrix   $M_3$, then a similar  estimate is harder to obtain, but one can show that 
$$ \Vert L^n(v)\Vert \ge \frac{ |\lambda_1|^{n+1}} {\sqrt{2\lambda_1^2 +n^2}} \Vert v\Vert. $$ 
To see this, one bounds the operator norm of the matrix $M_3^{-n}$ by its
Hilbert-Schmidt norm. We leave the details to the reader.

Since all norms on $\R^2$ are comparable, it follows that
$$|L^n(v)| \ge \rho^{n} |v|$$ 
for all sufficiently large $n$ independent of $v\in \R^2$ with 
$\rho=|\lambda_1|^{1/2}>1$.
\end{proof}

Recall that if $G$ is a planar crystallographic group, then we
say that a continuous
map $\Theta\: \R^2\ra S^2$ is 
induced\index{map!induced by group action}\index{induced by group action}
by $G$ if $\Theta(u)=\Theta(v)$ for $u,v\in \R^2$ if and only if
there exists $g\in G$ such that $v=g(u)$ (see
Section~\ref{sec:appquotmaps}). In this case, $G$ is necessarily
of non-torus type (see Theorem~\ref{thm:G_signature}) and
$\Theta$ is essentially the quotient map
$\Theta\: \R^2 \ra \R^2/G\cong S^2$ (see the discussion after 
Proposition~\ref{prop:G_yields_T}). 

\begin{lemma} 
  \label{lem:unifcontTheta} 
  Let $G$ be a planar crystallographic group  and 
    $\Theta\: \R^2 \ra S^2$ be a continuous map induced by $G$. Suppose 
 $K_n\sub \R^2$ is a connected set for $n\in \N$. Then 
 $$\displaystyle \lim_{n\to \infty} \diam(K_n)=0 \text{ if and only if } 
  \displaystyle \lim_{n\to \infty} \diam(\Theta(K_n))=0. $$   
 \end{lemma}

Here $\diam (K_n)$ is  the Euclidean diameter of $K_n$, and 
$\diam(\Theta(K_n))$  the diameter of $\Theta(K_n)$ with respect to some base  metric $d$ on $S^2$. 

\begin{proof} ``$\Rightarrow$'' For this implication it is enough to show that 
$\Theta$ is uniformly continuous on $\R^2$. This follows from the fact that 
$\Theta$ is induced by  $G$ and that $G$ acts isometrically and  cocompactly on $\R^2$;
indeed, we can find a compact fundamental domain $F\sub \R^2$ for the action 
of $G$ on $\R^2$. Now suppose $x,y\in \R^2$ and $\delta\coloneqq |x-y|$ is small.
 Then there exists 
$g\in G$  such that $g(x)\in F$. If $\delta$ is small enough, then $g(x), g(y)\in U$, where $U$ is a compact neighborhood of $F$. Since $\Theta$ is uniformly continuous on $U$, and $|g(x)-g(y)|=|x-y|=\delta$, it follows that 
$$ d(\Theta(x), \Theta(y) )= d(\Theta(g(x)), \Theta(g(y) )$$ 
is small only depending on $\delta$. The uniform continuity of $\Theta$ follows. 

\smallskip
``$\Leftarrow$''
 We argue by contradiction and assume
  that the statement is false. Then
  there exist connected sets $K_n\sub \R^2$ 
   such $\diam(\Theta(K_n))\to 0$ as $n\to \infty$,  but  $\diam(K_n)\ge \eps_0$ for $n\in \N$, where  $\eps_0>0$.

  We pick a point $x_n\in K_n$ for $n\in \N$. If we replace each
  set $K_n$ with its image  $K'_n=g_n(K_n)$ for suitable
  $g_n\in G $ (note that $\diam(K'_n)=\diam(K_n)$ and
  $\Theta(K'_n)=\Theta(K_n)$), and pass to a subsequence if
  necessary, then we may assume that the sequence $\{x_n\}$
  converges, say $x_n\to x\in \R^2$ as $n\to \infty$.
 
 Let $p\coloneqq \Theta(x)$.  Then the set $\Theta^{-1}(p)$ is equal to the orbit $Gx$ of $x$ under $G$.  Since the action of $G$ on $\R^2$ is properly discontinuous, the set 
 $\Theta^{-1}(p)=Gx$ has no limit point in $\R^2$.
 Since $G$ also acts cocompactly on $\R^2$, this  implies that the distance of distinct points in $\Theta^{-1}(p)$ is bounded away from $0$; so there exists a constant $m>0$ such that $|u-v|\ge m$ whenever $u,v\in \Theta^{-1}(p)$ and $u\ne v$.  
 
 Pick a constant $c$ with $0<c<\min\{\eps_0/2, m\}$.  The set $K_n$ is connected, and has diameter $\diam(K_n)\ge \eps_0> 2c$. Hence $K_n$ cannot be contained in the disk $\{z\in \R^2: |z-x_n|<c\}$, and so  it meets the circle $\{z\in \R^2: |z-x_n|=c\}$. It follows  that  there exists a
 point $y_n \in K_n$ with $|x_n-y_n|=c$. By passing to another subsequence if necessary, we may assume that the sequence $\{y_n\}$ converges, say $y_n\to y\in \R^2$ as $n\to \infty$. 
 Then $|x-y|=c<m$. Note that $\Theta(x_n), \Theta(y_n)\in \Theta(K_n)$ for $n\in \N$, 
and $\diam(\Theta(K_n))\to 0$ as $n\to \infty$. 
So
$$ p=\Theta(x)=\lim_{n\to \infty} \Theta(x_n)=\lim_{n\to \infty} \Theta(y_n) =\Theta(y), $$ 
and  $x,y\in \Theta^{-1}(p)$. Since $|x-y|=c>0$, we have $x\ne y$. So $x$ and $y$ are two distinct 
points in $\Theta^{-1}(p)$ with $|x-y|=c<m$. This contradicts the choice of $m$, and the claim follows. 
 \end{proof}

\begin{proof}[Proof of Proposition~\ref{prop:expLattType}] 
In the proof  metric notions on $\R^2$ will  refer to
the Euclidean metric.  

Let $f\: S^2 \ra S^2$ be the given Latt\`es-type map, and $A$, $\Theta$, $G$ be as in
Definition~\ref{def:Lattestype}. Then $\Theta\: \R^2\ra S^2$ is a branched covering map induced by the  crystallographic group $G$. We know that here $G$ is not isomorphic to $\Z^2$, because the quotient space $\R^2/G\cong S^2$ is not a torus. So by Proposition~\ref{prop:G_yields_T} there exists a holomorphic branched covering map $\widetilde \Theta\: \R^2\cong \C\ra \CDach$   induced by $G$, and a   unique homeomorphism 
$\varphi \: S^2 \ra \CDach$ such that  $\widetilde \Theta= \varphi \circ \Theta$ (note that the roles of the maps $\widetilde \Theta$ and $\Theta$  are reversed in Proposition~\ref{prop:G_yields_T}). 

We now conjugate $f$ by $\varphi$ to obtain a Thurston map $\widetilde f\coloneqq \varphi\circ f \circ \varphi^{-1}$ defined on $\CDach$. Then $\widetilde f$ is a Latt\`es-type map with   the triple $A$, $\widetilde \Theta$, $G$  as in
Definition~\ref{def:Lattestype}. Note that the affine map $A$ has not changed here and that $f$ is expanding if and only if $\widetilde f$ is expanding. In other words,  in order to prove the statement, we 
can make the additional assumptions that 
 the Latt\`es-type map $f$ is defined on $\CDach$ and the map $\Theta\: \C \ra \CDach$ is holomorphic. 

Then $\Theta\: \C \ra
\CDach$  is  the universal orbifold covering map of the parabolic orbifold $\mathcal{O}_f=(\CDach,\alpha_f)$ of $f$ (see
Proposition~\ref{prop:noperparaLTM}, Corollary~\ref{cor:orbofLattty},  and Theorem~\ref{thm:orbunifparacas}). 

Let $\omega$ be the canonical orbifold
    metric of $\mathcal{O}_f$.\index{canonical orbifold!metric}\index{metric!canonical orbifold} 
Since $\mathcal{O}_f$ is parabolic, $\om$ 
 is essentially the
push-forward of the Euclidean metric  on $\R^2$ by $\Theta$
(see Section~\ref{sec:expratThmaps} and Section~\ref{sec:orbif-assoc-thurst}). The metric  $\om$ is a length metric that induces the standard   topology on $\CDach$ (so it can be used as a base metric on $\CDach$ as in Section~\ref{sec:defin-expans-revis} or Lemma~\ref{lem:unifcontTheta}) and  the map $\Theta\: \R^2\ra \CDach$ is a path isometry in the sense that 
$$ \length(\alpha) = \length_\omega(\Theta \circ \alpha)$$ for each path 
$\alpha$ in $\R^2$.

 If  $L=L_A$ is the linear part of $A$, then  the map $L^n$ is the linear part of $A^n$ for each $n\in \Z$. So $L^n$ and  $A^n$ only differ by a translation. Since translations are isometries, it follows that 
\begin{equation}\label{eq:lengthALalp}
\length( L^n\circ \alpha)=\length(A^n\circ \alpha) 
\end{equation} 
 for all  $n\in \Z$ whenever  $\alpha$ is  a path in $\R^2$. If
 $\gamma\coloneqq \Theta\circ \alpha$, then  
 \begin{equation*}
   f^n\circ \ga
   =  
   f^n\circ \Theta \circ \alpha 
   = 
   \Theta \circ A^n  \circ \alpha 
 \end{equation*}
for $n\in \N$.  
Since $\Theta$ is a path isometry, we conclude that 
\begin{equation*}
  \length_\omega(f^n\circ \ga) 
  = 
  \length (A^n\circ \alpha )
  =
  \length (L^n\circ \alpha )
\end{equation*}
for  $n\in \N$. 

Now suppose that $L$ is expanding. Then Lemma~\ref{lem:explinmap} implies  that there exist $N\in \N$  and a constant $\rho>1$ such that 
$$ \length(L^N \circ \alpha) \ge \rho \length (\alpha)$$ 
for all paths $\alpha$ in $\R^2$. If $\gamma$ is an arbitrary path in $\CDach$, then it has a lift by the branched covering map $\Theta$ (see Lemma~\ref{lem:liftsofpathsbranched}); so it can be written in the form $\ga=\Theta\circ \alpha$, where $\alpha$ is a path in $\R^2$.  Hence 
\begin{align*} 
 \length_\omega(f^N\circ \ga)
  &= \length (L^N\circ \alpha ) 
 \ge  \rho \length (\alpha)
  \\ 
  &=  \rho \length_\omega (\Theta\circ \alpha)
    =\rho \length_\om (\ga).
\end{align*} Lemma~\ref{lem:length_exp} implies that $f^N$ is an expanding Thurston map. Hence $f$ is expanding (see Lemma~\ref{lem:Thiterates}). 

 To prove the converse, we assume that $f$ is expanding, but $L$ is not. 
 Then one of the eigenvalues of $L$ has absolute value $\le 1$. 
 Since the product of these eigenvalues is equal to 
  $\det(L) = \deg(f)\ge 2$ (see Lemma~\ref{lem:deglatttype}), it follows 
  from  the considerations in the beginning  
 of the proof of Lemma~\ref{lem:explinmap} that both eigenvalues of $L$ are real.  So  $L$ has a real eigenvalue $\lambda$  with $ |\lambda|\le 1$. Note that $\lambda \ne 0$, because $L$ is invertible. 
 
 Then  there exists $u\in \R^2$ with $|u|=1$ such that 
 $L(u)=\lambda u$. Let  $\alpha$ be  the parametrized line segment joining $0$ and $u$,  and define $\alpha_n\coloneqq A^{-n}\circ \alpha$ for $n\in \N$.  Then 
 \begin{align}\label{eq:lowerdiambddan} 
 \diam(\alpha_n)&= \diam(A^{-n}\circ \alpha)= \length(A^{-n}\circ \alpha)\\& = \length(L^{-n}\circ \alpha)=\frac{1}{|\lambda|^n}
  \length(\alpha)\notag \\ &\ge \length(\alpha)=1\notag 
  \end{align} 
  for all $n\in \N$.

Suppose for  $n\in \N$ the path  $\ga_n$ is  a lift of  some  path $\ga$ in $\CDach$ by  $f^n$. Since $f$ is expanding,   we then 
have $\diam_\om(\ga_n)\to 0$ as $n\to \infty$. 
Indeed,  if $\ga$ is a path whose diameter is less than 
the Lebesgue number $\delta>0$ of an open cover $\mathcal{U}$ as in 
Proposition~\ref{prop:expequivexp}~\ref{item:exp_cover}, then there exists $U\in \mathcal{U}$ such that $\ga\sub U$. Then $\ga_n$ lies in a connected component of $f^{-n}(U)$ and so 
$$ \diam_\om(\ga_n) \le \mesh(f^{-n}(\mathcal{U}))\to 0$$ 
as $n\to \infty$.  The statement $\diam_\om(\ga_n)\to 0$ as $n\to \infty$
remains true for arbitrary paths $\ga$, because we can break $\ga$ up into finitely many paths of diameter $<\delta$. 
(We will later see that with respect to a {\em visual metric} for $f$ 
(see Chapter~\ref{cha:visual-metrics})   the diameters of lifts of any path by $f^n$ actually shrink to $0$ exponentially fast as $n\to \infty$
 (see Lemma~\ref{lem:liftpathshrinks})).
 
 We apply this to   $\ga\coloneqq \Theta\circ \alpha$, and 
$ \ga_n\coloneqq  \Theta\circ\alpha_n$ for $n\in \N$. The path $\ga_n$ is a lift of $\ga$ by $f^n$, because 
$$ f^n\circ \ga_n =   f^n \circ \Theta\circ A^{-n} \circ \alpha= \Theta \circ A^n \circ A^{-n}  \circ \alpha=  \Theta   \circ \alpha=\ga. $$ 
 
We now obtain a contradiction from  
 Lemma~\ref{lem:unifcontTheta}, because the sets   $\alpha_n$ 
 are connected and 
 $$\diam_\om(\ga_n)=\diam_\om (\Theta(\alpha_n))\to 0 $$ 
 as $n\to \infty$, but $\diam(\alpha_n)\ge 1$ for all $n\in \N$ by 
  \eqref{eq:lowerdiambddan}.
  
It follows that    $L$ is expanding if $f$ is. Together with the first part of the proof, we conclude that $f$ is expanding as a Thurston map if and only if $L=L_A$ is expanding as a linear map. 
 \end{proof}

We finish this chapter by giving an example of a Thurston map
that is eventually onto, 
but not expanding. The example is due to K.~Pilgrim.
 
 \begin{ex} 
   \label{ex:non-expanding-lattes}
   Let $G$ be the crystallographic group consisting of all maps
   $g$ of the form $$u\in \R^2 \mapsto g(u)=\pm u+\ga, $$  where
   $\ga \in \Z^2$. So $G$ is of type $(2222)$  (see
   Theorem~\ref{thm:G_signature}). 
Let $\Theta\: \R^2 \ra \R^2/G\cong S^2$ be the quotient map. 

We consider  the matrix 
\begin{equation*}
 A =  \left(
    \begin{array}{cc}
      4 & 2
      \\
      2 & 2
    \end{array}
  \right).
\end{equation*}
and the map $A\: \R^2\ra \R^2$, $u\in \R^2\mapsto Au$ given by left-multiplication of  $u\in \R^2$ (written as a column vector) by the matrix $A$. For simplicity, here (and also below)   we not not distinguish in our notation between a matrix and the linear map it induces on $\R^2$ by left-multiplication. 

The map $A$ has the form \eqref{eq:Lattestype2222}. The 
linear part $L=L_A$ of $A$ agrees with $A$. So 
 $L$ is also represented by the matrix $A$. Since
$\det(L_A)=\det(A)\ge 2$, 
we know that
the map $A$ induces a
Latt\`es-type map $f\: S^2\ra S^2$ on the quotient $S^2=\R^2/G$
according to Proposition~\ref{prop:2222}.
We claim that {\em $f$ is not expanding}, but \emph{eventually
  onto.} 
\index{eventually onto} 

Recall that the latter property means that for any
non-empty open set $U\sub S^2$ there is an iterate $f^n$ such that $f^n(U)=S^2$
(see also Lemma~\ref{lem:event_onto}). 

The map  $L=A$ has the eigenvalues $\lambda_1= 3-\sqrt{5}$ and  
  $\lambda_2=3+\sqrt{5}$. Since $|\lambda_1|<1$, the map $f$ is not expanding by Proposition~\ref{prop:expLattType}.     
  
Now consider the linear maps $B$ and $C$  given by left-multipli\-cation of $u\in \R^2$ with  the matrices  
  $$ B=
  \left(
    \begin{array}{cc}
      2 & 1
      \\
      1 & 1
    \end{array}
  \right)   \text{ and }     C=
  \left(
    \begin{array}{cc}
      2 & 0
      \\
      0 & 2
    \end{array}
  \right), 
  $$ respectively. Then $A=B\circ C=C\circ B$. The maps $B$ and $C$ again have the form \eqref{eq:Lattestype2222} and so descend to maps $g\: S^2\ra S^2$ and $h\: S^2 \ra S^2$ 
  respectively. Note that $\det( B)=1$, and so $g$ is a homeomorphism with an inverse induced by $B^{-1}$ (see Proposition~\ref{prop:2222}). 
  
  These maps satisfy 
  $f=g\circ h=h\circ g$. Indeed, note 
  that $f\circ \Theta = \Theta \circ A$, $g\circ
\Theta = \Theta \circ B$, and $h\circ \Theta = \Theta \circ
C$. Thus
\begin{equation*}
  f\circ \Theta 
  =  \Theta \circ A=
  \Theta \circ B \circ C 
  = 
  g\circ \Theta \circ C
  =
  g \circ h \circ \Theta.
\end{equation*}
Since $\Theta$ is surjective, it follows that $f= g\circ h$. A
similar  argument shows that $f=h\circ g$.
 Since $f=g\circ h=h\circ g$, we have  
  $ f^n = g^n\circ h^n$ for all $n\in \N$.  
  
  Now let $U\sub S^2$ be an arbitrary non-empty open set. 
  Then $V\coloneqq \Theta^{-1}(U)$ is a non-empty open set in $\R^2$. Since 
  $C^n(V)=2^n V$ this set  $C^n(V)$ will contain arbitrarily large disks if $n$ is sufficiently large. In particular, there exists $n\in \N$ such that 
  $C^n(V)$ contains a translate $\ga+R$ with $\ga\in \Z^2$ 
 of the  rectangle $R=[0,1]\times [0,1/2]$, which  is  a
 fundamental domain (see Section~\ref{sec:appquotmaps}) for the action of $G$. For  such $n$ we have 
  $$ S^2=\Theta(R)=\Theta(\ga+R)=\Theta(C^n(V))=h^n(\Theta(V))=h^n(U), $$ 
  and so, since $g$ is a homeomorphism,  
  $$ f^n(U)=(g^n\circ h^n)(U)=g^n(S^2)=S^2. $$ 
  This shows that $f$ is eventually onto. 
\end{ex}

\ifthenelse{\boolean{singlechapter}}{

\chapter{Thurston maps with two or three postcritical points}
\label{cha:thurston-maps-23}

In this chapter we investigate Thurston maps $f \: S^2\ra S^2$ with a postcritical set consisting of two or three elements (note that we always have $\#\post(f)\ge2$ by Corollary~\ref{cor:post012}).  The considerations here are not essential for our main story and may safely be skipped by the impatient reader. 

For starters it is easy to classify all
Thurs\-ton maps with two postcritical points up to Thurston
equivalence.

\begin{prop}
  \label{prop:post2} 
 A Thurston map   $f\:S^2\ra S^2$ 
  with
  $\#\post(f)=2$ is  Thurston equivalent to a power map
  $z\mapsto z^n$ on $\CDach$, where $n\in \Z\setminus\{-1,0,1\}$.
\end{prop} 
This will be proved in Section~\ref{sec:proofs-thurst-equiv}.
%
%
%
In  case $\#\post(f)=3$, the relation to rational Thurston maps 
is clarified by the following statement.

\begin{theorem}\label{thm:3postrat}
Let  $f\:S^2\ra S^2$ be  a Thurston map such that  $\#\post(f)=3$. Then the following statements are true:
  \begin{enumerate}
  \item 
    \label{item:3post_rat}
    $f$ is Thurston equivalent to a rational Thurston map.
    \index{Thurston map!rational}
  \item
    \label{item:3post_exp}
   If $f$ is expanding, then $f$ is topologically
    conjugate to a rational Thurston map if and only if $f$ has no
    periodic critical points.   
  \end{enumerate}
\end{theorem}
Part \ref{item:3post_rat} of this statement is essentially a trivial case of  Thurston's  characterization of rational maps
given in Theorem~\ref{thm:Thurston}.  
We will present a proof in
Section~\ref{sec:proofs-thurst-equiv}.  
Part~\ref{item:3post_exp} easily follows if this is combined with some of our other results.

In Chapter~\ref{cha:lattes-lattes-type} we considered  Latt\`{e}s and
 Latt\`{e}s-type maps.
These are  Thurston maps $f\: S^2 \ra S^2$  with a parabolic orbifold and no 
periodic critical points (see Proposition~\ref{prop:noperparaLTM}).  The case when  $f$ has a parabolic orbifold, but also periodic critical points, is very special. Then the  signature of $f$ must be  $(\infty,\infty)$ or  $(2,2,\infty)$ as follows from Propositions~\ref{prop:parabolicOf} and \ref{prop:otherramprops}. These maps can easily be classified up to Thurston equivalence.

\begin{theorem}
  \label{thm:para_Th_poly}
  Let   $f\colon S^2\to S^2$ be   a Thurston map. Then $f$ has
  signature 
  \index{Thurston map!parabolic}
  \index{orbifold!parabolic}
  \index{parabolic!orbifold}
  \begin{enumerate}
  \item 
    \label{item:Th_poly_inf_inf}
    $(\infty,\infty)$ if and only if  $f$ is Thurston equivalent to a power map $z\mapsto
    z^n$ on $\CDach$, where $n\in \Z\setminus \{-1,0,1\}$;
  \item 
    \label{item:Th_poly_22inf}
    $(2,2,\infty)$ if and only if $f$ is Thurston equivalent to $\chi$ or $-\chi$,
     where $\chi=\chi_n$ is a 
    \emph{Chebyshev polynomial}\index{Chebyshev polynomial}
    of degree $n\in \N\setminus\{1\}$.  
  \end{enumerate}
\end{theorem}

  The case of signature $(\infty,\infty)$ is essentially already
  covered by Proposition~\ref{prop:post2}.  The proof of
  Theorem~\ref{thm:para_Th_poly} is  given in
  Section~\ref{sec:parab-thurst-polyn}, where we will  also
  review the definition of  Chebyshev polynomials (see also
  \cite{Ri90}).

\section{Thurston equivalence to rational maps} 
\label{sec:proofs-thurst-equiv}

We begin by  looking at  Thurston maps $f$ with $\#\post(f)=3$.

\begin{proof}[Proof of Theorem~\ref{thm:3postrat}] 

  \ref{item:3post_rat} 
  Let $f\colon S^2\to S^2$ be a Thurston with three
  postcritical points, which we denote by $p_0,p_1,
  p_\infty$. Let $h_0\colon S^2 \to \CDach$ be an
  orientation-preserving homeomorphism. By postcomposing $h_0$ with a
  suitable M\"{o}bius transformation, we may assume that
  $h_0(p_0)=0$, $h_0(p_1) = 1$, and $h_0(p_\infty) =
  \infty$. We can think of the map $h_0$ as a global chart
  on  $S^2$; by this chart $S^2$  carries a  conformal structure.  

 If we pull-back this conformal structure by the map $f$, then 
 we obtain another  conformal structure on $S^2$. The map $f$ is then holomorphic with respect to
  these two conformal structures on $S^2$.  To be more precise, 
 Corollary~\ref{cor:brcovratup} gives 
  the existence of  an orientation-preserving homeomorphism $h_1\colon
  S^2\to \CDach$ and a rational map $R\colon \CDach\to \CDach$
  such that $h_0 \circ f = R \circ h_1$.
 On a more intuitive level, $R$  is the representation of the holomorphic map $f$ if  we use suitable global charts. 
  
   Again by postcomposing
   with a suitable M\"{o}bius transformation, we may assume
  that $h_1(p_0)= 0$, $h_1(p_1) = 1$, and $h_1(p_\infty) =
  \infty$. 

  Two orientation-preserving homeomorphisms on $S^2$ that agree
  on a set $P\subset S^2$ containing at most three points are
  isotopic rel.\ $P$ (see Lemma~\ref{lem:homeo}). Thus $h_0$ and
  $h_1$ are isotopic rel.\ $\post(f) = \{p_0,p_1,
  p_\infty\}$. It now follows from  Lemma~\ref{lem:T-eq_crit_post} 
  that $R$ is a Thurston map, and it is clear that $f$ and $R$ are Thurston equivalent.

\smallskip
\ref{item:3post_exp}
 Now suppose in addition that $f$ is expanding and has no periodic
critical points. Since the latter   condition is invariant under
Thurston equivalence, the rational Thurston  map $R$ constructed 
above will then not have periodic critical points either, and  is hence
expanding by Proposition~\ref{prop:rationalexpch}. Therefore, the maps
$f$ and $R$ are topologically conjugate by a general result that
will be proved later 
(see  Theorem~\ref{thm:exppromequiv}). 

Conversely, if $f$ is expanding and topologically conjugate to a rational 
 map $R$, then $R$ is an expanding Thurston map. Hence $R$ has no periodic critical points by Proposition~\ref{prop:rationalexpch}, which implies that $f$ cannot have periodic critical points either. 
 \end{proof}


  


Let us now consider the case $\#\post(f)=2$. 

\begin{proof}[Proof of Proposition~\ref{prop:post2}]
  Assume $f\colon S^2\to S^2$ is a 
  Thurston map with
  $\#\post(f)=2$. We want to show that $f$ is Thurston
  equivalent to the map $z\mapsto z^n$ on $\CDach$, where
  $n\in \Z\setminus\{-1,0,1\}$.

  The proof that $f$ is Thurston equivalent to a rational map
  $R\colon \CDach \to \CDach$ is identical to the proof of
  Theorem~\ref{thm:3postrat}~\ref{item:3post_rat}. The only
  adjustment is that the maps $h_0\colon S^2\to \CDach$ and
  $h_1\colon S^2 \to \CDach$ have to agree on the set
  $\post(f)$, which now contains just two points, instead of
  three. Again this is achieved by postcomposing with  suitable
  M\"{o}bius transformations. 

  The rational map $R$ has also  two postcritical points, and so
  we may assume that $\post(R)=\{0, \infty\}$ (by conjugating
  $R$ with a suitable M\"{o}bius transformation). Let
  $A= R^{-1}(0)$ be 
the set of zeros 
and $B= R^{-1}(\infty)$
  be the set of poles of $R$.  From Lemma~\ref{lem:postf-2}
  it follows that $A\cup B=\post(R) = \{0,\infty\}$. This
  implies that $R(z) = cz^n$ for $z\in \CDach$, where
  $c\in \C\setminus\{0\}$ and $n\in \Z \setminus \{0\}$. Here
  actually $n\in \Z \setminus \{-1, 0, 1\}$, because $R$ is a
  Thurston map and hence not a homeomorphism.  So in particular,
  $n\ne 1$, which implies that by conjugating $R$ with an
  auxiliary map of the form $z\mapsto \alpha z$,
  $\alpha\in \C\setminus\{0\}$, if necessary, we may assume that
  $c=1$. The claim follows.
  \end{proof}

\section{Thurston maps with signature $(\infty,\infty)$ or $(2,2,\infty)$}
\label{sec:parab-thurst-polyn}

In this section we consider Thurston maps $f$ with a parabolic
orbifold and periodic critical points. Equivalently, the
associated orbifold $\OC_f$ has signature $(\infty,\infty)$ or
$(2,2,\infty)$. These
maps together with Latt\`es and 
Latt\`{e}s-type maps considered
in Chapter~\ref{cha:lattes-lattes-type} cover all cases of
Thurston maps with a parabolic orbifold (see
Proposition~\ref{prop:noperparaLTM}). 

The case when the signature is $(\infty,\infty)$ has
already been treated  in Proposition~\ref{prop:post2} (see also
Lemma~\ref{lem:postf-2}). It remains to consider  the case of signature $(2,2,\infty)$. 
Our presentation  follows
\cite{Mi06}.

\begin{lemma}
  \label{lem:22infty}
  Let $f\colon S^2\to S^2$ be a Thurston map with signature
  $(2,2,\infty)$. Then $f$ is a Thurston polynomial.
\end{lemma}

\begin{proof}
  Suppose the  signature of  $f$ is  $(2,2,\infty)$ and let
  $p\in S^2$ be the unique point with $\alpha_f(p)=\infty$. 
  Proposition~\ref{prop:parabolicOf} implies that  
  $\alpha_f(q)=\infty$  for each point $q\in f^{-1}(p)$.  Thus, 
  $q=p$ and so $p$ is completely invariant. Therefore,  $f$ is a Thurston
  polynomial.
\end{proof}

We call a Thurston polynomial $f$ with a parabolic orbifold a
{\em parabolic Thurston polynomial}.\index{parabolic!Thurston polynomial}\index{Thurston!polynomial!parabolic} 
Then $f$ has signature
$(\infty, \infty)$ or $(2,2,\infty)$. Conversely, if $f$ is a
Thurston map with signature $(2,2,\infty)$, then $f$ is a
parabolic Thurston polynomial by Lemma~\ref{lem:22infty}. If $f$
has signature $(\infty, \infty)$, then $f$ or $f^2$ is a
parabolic Thurston polynomial as follows from
Proposition~\ref{prop:post2}. A 
classification of parabolic Thurston polynomials is obtained
from Theorem~\ref{thm:para_Th_poly} which we will prove below.
 
Lemma~\ref{lem:poly_not_exp} implies that 
 a parabolic Thurston polynomial  $f$ cannot be
expanding.    Moreover,  if $f$ is rational and
suitably normalized, then $f$ is a polynomial.

 The polynomials that appear here are very special, namely power maps $z\mapsto z^n$ or 
Chebyshev polynomials. By definition  the {\em Chebyshev
  polynomial}\index{Chebyshev polynomial|textbf}  $\chi_n$  for $n\in \N_0$ is the unique polynomial 
such that 
\begin{equation}\label{eq:functCh}
 \cos(n u) =\chi_n(\cos u), \quad u\in \C. 
 \end{equation} 
Thus, the first Chebyshev polynomials are $\chi_0(z)=1$,
$\chi_1(z)=z$, $\chi_2(z) = 2z^2 -1$, and $\chi_3(z) =
4z^3-3z$. They satisfy the recurrence relation 
\begin{equation*}
  \chi_{n+1}(z)=2z\chi_n(z) - \chi_{n-1}(z)
\end{equation*}
for  $n\in \N$. This implies  that
$\deg(\chi_n)=n$. Since 
\begin{align*}
  \chi_n(-\cos u) &= \chi_n(\cos(u + \pi)) = \cos(nu + n\pi)\\
  &=(-1)^n  \cos(nu)= (-1)^n  \chi_n(\cos u),
\end{align*}
it follows that $\chi_n$ is an even function when $n$ is even, and $\chi_n$
is an odd function when $n$ is odd.  

To find the critical points of $\chi_n$,  we differentiate  \eqref{eq:functCh} and  obtain 
$$- n\sin (n u) =- \chi'_n(\cos u) \sin u, \quad u\in \C. $$
For $n\ge 1$ the left  hand side has a simple zero whenever $nu\in \Z\pi$. In particular, if we consider $u_k=k\pi/n$ and  the corresponding points $z_k=\cos(k\pi /n)$ for $k=0, \dots, n$, we see that 
$\chi_n'$ has a simple zero at $z_k$ for $k=1, \dots, n-1$.
Moreover, $\chi'_n(z_0)=\chi'_n(1)\ne 0$ and  $\chi'_n(z_n)=\chi'_n(-1)\ne 0$,
because $\sin(u_0)=\sin(u_n)=0$. 
 Since $\chi_n$ has at most $n-1$ critical points and $u\mapsto \cos(u)$ is injective on 
$[0, \pi]$, the points   $z_k$ for $k=1, \dots, n-1$ are   the  distinct critical points of $\chi_n$ and  $\deg(\chi_n, z_k)=2$. 
By \eqref{eq:functCh} we have  $\chi_n(z_k)=(-1)^k$ 
for $k=1, \dots, n-1$,   $\chi_n(-1)=\chi_n(1)=1$ for $n$ even, and  $\chi_n(\pm1)=\pm1$ for $n$ odd. This implies  that $\chi=\chi_n$ for $n\ge 2$ is a postcritically-finite polynomial
with $\post(\chi)=\{-1, 1, \infty\}$. For the ramification function $\alpha_\chi$ of $\chi$ we conclude from this analysis that $\alpha_\chi(-1)=\alpha_\chi(1)=2$ and $\alpha_\chi(\infty)=\infty$. 

We will also have to consider  the polynomial $\chi=-\chi_n$ for $n\ge 2$. The same considerations  show that again $\post(\chi)=\{-1, 1, \infty\}$, 
$\alpha_\chi(-1)=\alpha_\chi(1)=2$, and $\alpha_\chi(\infty)=\infty$. 
It follows  that $\chi=\pm\chi_n$ for $n\ge 2$ is a postcritically-finite polynomial whose orbifold has signature $(2,2, \infty)$. 

The postcritical
points $-1$  and $1$ are mapped by $\chi=\pm\chi_n$
as indicated in the following diagrams: if  $\chi=\chi_n$, 

\begin{align}\label{eq:-11map1} 
  \xymatrix {
    -1 \ar[r] & 1 \ar@(r,u)[]
  }
 & \quad \ \text{ for } n \text{ even and }
   \xymatrix @R=1pt{
    -1 \ar@(r,u)[] & 1 \ar@(r,u)[] 
  } 
   \quad \ \text{for } n \text{ odd;}
\end{align}
and if  $\chi=-\chi_n,$

\begin{align}\label{eq:-11map2} 
  \xymatrix @R=1pt{
    1 \ar[r] & -1 \ar@(r,u)
  }
   \quad \ \text{ for } n \text{ even and }
    \xymatrix @R=1pt{
    -1 \ar@/^1pc/[r] & 1 \ar@/^1pc/[l] 
  } 
  & \ \quad \text{ for } n \text{ odd.}
\end{align}
The diagrams when $n$ is even are similar in both cases, because the maps 
involved are topologically conjugate. 
Indeed, suppose $n\in \N$ is even,  and let $\tau(z)=-z$ for $ z\in
\C$. Since $\chi_n$ is an even function, it follows that 
$$(\tau\circ \chi_n \circ\tau^{-1})(z)=-\chi_n(-z)=-\chi_n(z)$$ 
for $z\in \C$, 
which implies that $-\chi_n=\tau\circ \chi_n\circ \tau^{-1}$ for $n\in \N$ even. 

When $n\in \N$ is odd, the maps $\chi_n$ and $-\chi_n$  are not
topologically conjugate (or Thurston equivalent) as  the
above diagrams show: all postcritical points are fixed points for $\chi_n$, but not 
for $-\chi_n$. 

After this preliminary discussion, we can now prove Theorem~\ref{thm:para_Th_poly}. 

\begin{proof}[Proof of Theorem~\ref{thm:para_Th_poly}]
Statement~\ref{item:Th_poly_inf_inf} immediately follows from 
  Proposition~\ref{prop:post2} and 
  Proposition~\ref{prop:Thequivsamesig}.

   To prove \ref{item:Th_poly_22inf}, let $f\colon S^2\to
  S^2$ be a Thurston map with 
  signature $(2,2,\infty)$. Then $\#\post(f)=3$, and so by Theorem~\ref{thm:3postrat}~\ref{item:3post_rat} the map $f$ is Thurston equivalent to a rational map, necessarily with the same signature.
So  we may assume that $f$ is a rational map on $S^2=\CDach$ to begin with. Moreover, by conjugating  the map with a suitable M\"obius transformation, we may assume that 
$\post(f)=\{-1,1,\infty\}$ and that 
for the ramification function $\alpha_f\: \CDach\ra \widehat \N$ of $f$ we have 
$\alpha_f(-1)= \alpha_f(1)=2$ and
$\alpha_f(\infty)=\infty$. By the argument in the proof of  Lemma~\ref{lem:22infty} we see  that then $f$ is a polynomial.

The map $f$ has a parabolic orbifold, and so Proposition~\ref{prop:parabolicOf} implies that  \begin{equation}\label{eq:Chpara}
\alpha_f(z)\cdot\deg_f(z)=\alpha_f(f(z))
\end{equation}
for all $z\in \CDach$. 
Equation 
 \eqref{eq:Chpara}  shows  that we have $z\in f^{-1}(\{-1, 1\})$ if  and only if $z\ne \infty$ and 
 one of the factors on the left-hand side of \eqref{eq:Chpara} is different from $1$.   Then one 
  factor is  equal to $1$ and the other equal to $2$.   We conclude that 
 $$f^{-1}(\{-1,1\})= \{-1,1\}\cup \crit(f)\setminus\{\infty\}$$ with   
 $\deg_f(-1)=\deg_f(1)=1$ and $\deg_f(z)=2$ for $z\in \crit(f)\setminus\{\infty\}$.
 This implies that   the polynomials $1-f(z)^2$ and
 $(1-z^2)f'(z)^2$ have the same zeros of exactly 
the same orders.  
If we define $n=\deg(f)\ge 2$, then comparison of the highest-order coefficient gives
\begin{equation} \label{eq:ChODE1}
n^2(1-f(z)^2)=(1-z^2)f'(z)^2\end{equation}
for $z\in \C$. It is well known and easy to prove that then $f=\pm \chi_n$. 

Indeed, to see this, consider the even entire function $g$ defined as $g(u)=f(\cos u)$ for $u\in \C$. Then \eqref{eq:ChODE1} leads to 
$$ g'(u)^2=n^2(1-g(u)^2)$$
for $u\in \C$. If we differentiate this equation, then we obtain the
linear ordinary differential equation
\begin{equation*}
  g''(u)+n^2g(u)=0,\quad u\in \C. 
\end{equation*}
It has the general solution $g(u)=c_1\cos(n u) +c_2 \sin(n u)$, $c_1,c_2\in \C$. Since $g$ is even, we must have $c_2=0$; moreover, $c_1=g(0)=f(1)\in \{-1,1\}$. Hence $g(u)=\pm \cos(nu)=f(\cos u)$ for $u\in \C$. This implies  $f=\pm \chi_n$. 

For the converse direction suppose that the Thurston map $f\: S^2 \ra S^2$ is Thurs\-ton equivalent to  $\chi=\pm \chi_n$ with $n\in \N \setminus\{1\}$. We have seen earlier in this section that $\chi$ has signature $(2,2,\infty)$. Hence $f$ has the same signature by
 Proposition~\ref{prop:Thequivsamesig}.\end{proof}

As we have seen in Chapter~\ref{cha:lattes-lattes-type},   Latt\`{e}s maps are related to crystallographic  groups $G$ acting on $\C$. Here $G$ contains a  
subgroup 
$\Gtr$
of translations isomorphic to a rank-$2$ lattice. We will now discuss how 
the maps $z\mapsto z^n$ for $n\in
\Z\setminus \{-1,0,1\}$  and $z\mapsto \pm\chi_n(z)$ for $n\in \N\setminus \{1\}$ can  be described in a similar
fashion.  According to Theorem~\ref{thm:para_Th_poly},  every Thurston map with signature
$(\infty,\infty)$ or $(2,2,\infty)$ is Thurston equivalent to
such a map. Since  the following considerations are fairly elementary, we will skip some details. 

Recall from Section~\ref{sec:cryst-groups-latt} that $\Isom(\C)$
denotes the group of all orien\-ta\-tion-preserving isometries of
$\C$ (equipped with the Euclidean metric) and $\Aut(\C)$ the
group of holomorphic automorphisms of $\C$. For a group
$G\subset \Isom(\C)$, we denote by $\Gtr$ the subgroup of $G$
consisting of translations in $G$, i.e., $\Gtr$ consists of all
maps $g\in G$ of the form $z\mapsto g(z) = z + \gamma$ 
with $\gamma\in \C$. If we denote 
by   $\Gamma \subset \C$ the set of all such $\gamma\in
\C$, then $\Gtr = \{z\mapsto z + \gamma : \gamma\in\Gamma\}$. If the action of $G$ on $\C$ is properly discontinuous, then $\Gamma$ is necessarily a discrete set in $\C$ and so a lattice.

We now
focus on the case that $\Gamma$ is a  rank-$1$ lattice, meaning that
$\Gamma$ spans a $1$-dimensional subspace of $\R^2\cong \C$.
The following lemma is closely related to
Theorem~\ref{thm:G_signature}. 

\begin{lemma}
  \label{lem:frieze_gps}
  Let $G\subset \Isom(\C)$ be a group such that the action
  of $G$ on $\C$ is properly discontinuous, and $\Gamma\subset
  \C$ as defined above is of rank $1$. Then $G$ is conjugate to
  one of the following groups $\widetilde{G}$ consisting of all
  $g\in \Isom(\C)$ of the form
  \begin{enumerate}
  \item[\upshape$(\infty\infty)$]
    $\quad\quad\quad z\mapsto g(z) = z+ k$, where $k\in \Z$;
  \item[\upshape$(2 2\infty)$] 
    $\quad\quad\quad z\mapsto g(z) = \pm z+ k$, where $k\in
    \Z$. 
  \end{enumerate}
\end{lemma}
As in Chapter~\ref{cha:lattes-lattes-type}, we say that $G$ and $\widetilde{G}$ are {\em conjugate} 
if there exists  $h\in \Aut(\C)$ such that $\widetilde{G}= h \circ G \circ h^{-1}$. 
We do not provide the (well-known)
proof of Lemma~\ref{lem:frieze_gps} here;  it  can be found in \cite{Ar}.

If for a  
group $G$ as in Lemma~\ref{lem:frieze_gps} the conjugate group 
$\widetilde{G}$ has the form  $(\infty \infty)$ or $(22\infty)$, then
we say that $G$ is of {type} $(\infty \infty)$ or 
type $(22\infty)$, respectively. 
 Again we are using Conway's orbifold notation. Clearly a group
of type $(\infty \infty)$ is isomorphic (as a group) to $\Z$,
and a group of type $(22\infty)$ is isomorphic to the
\emph{infinite dihedral group} $D_\infty=\Z 
\rtimes \Z_2$.

If  $G=\widetilde{G}$ is one of the groups in Lemma~\ref{lem:frieze_gps}, then its
subgroup of translations $\Gtr$ consists of the maps   of the
form $z\in \C\mapsto g(z) = z+ k$, where $k\in \Z$. The quotient 
$\C/\Gtr$  is an  \emph{(infinite) cylinder} (see below for a geometric justification
of this terminology). 

Clearly,
$\exp(2\pi \iu z) = \exp(2\pi\iu w)$ for $z,w\in \C$ if and only
if $w=z+k$ for some $k\in \Z$. This means that the map  
$\pi \: \C\ra \C^{*}\coloneqq \C\setminus \{0\}$ given by $\pi(z)=\exp(2\pi \iu z) $ for $z\in \C$ is induced by $\Gtr$; so we can identify 
 the cylinder $\C/\Gtr$ with  $\C^{*}$ and consider $\pi\: \C\ra \C^*\cong \C/\Gtr$ as the quotient map (see Corollary~\ref{cor:groupquot}). The cylinder 
 $\C^*$ plays a similar role for the groups in Lemma~\ref{lem:frieze_gps} as tori 
 obtained as quotients $\C/\Gtr$ of crystallographic groups $G$.
 
Now let  $f(z)=z^n$ with $n\in \Z\setminus \{-1,0,1\}$. We know  that the orbifold 
$(\CDach, \alpha_f)$ of $f$ has signature $(\infty, \infty)$ and two {\em punctures} (i.e., points $p$ with $\alpha_f(p)=\infty$)  
at $0$ and $\infty$. So if we remove these punctures from
$\CDach$ 
we obtain the set $\CDach_0=\C \setminus \{0\}=\C^*$ 
(see Section~\ref{sec:orbifolds-coverings} for a related discussion). The (holomorphic) universal orbifold covering map $\Theta\:\C \ra \CDach_0=\C^*$ is given by   $\Theta(z) = \exp(2\pi \iu z)$ for $z\in \C$. It is induced by the group $G=\widetilde G$ of the form $(\infty\infty)$ in Lemma~\ref{lem:frieze_gps}. The map $A\: \C\ra \C$ given by   $A(z)=nz$  for $z\in \C$  is $G$-equivariant.  If  we define   $\overline A(z)=z^n$ for $z\in \C^*$ and $\overline \Theta=\id_{\C^*}$, then we obtain the following commutative diagram:
 \begin{equation}
  \label{eq:inf_inf_diagram}
  \xymatrix{
    \C \ar[r]^{A(z) =n z} \ar[d]_{\pi} 
    \ar@/_2pc/[dd]_{\Theta(z)= \exp(2\pi \iu z)} 
    & \C \ar[d]^{\pi}\ar@/^2pc/[dd]^{\Theta(z)= \exp(2\pi \iu z)}
    \\
    \C^* \ar[r]^{\overline{A}}\ar[d]_{\overline{\Theta}} 
    & \C^*\ar[d]^{\overline{\Theta}}
    \\
    \C^*\ar[r]^{z\mapsto z^n} & \C^*\rlap{.}
  }
\end{equation}

It   is the  analog of  the diagram in \eqref{eq:holom_qu_diagram} that we obtained for Latt\`es maps based on Theorem~\ref{thm:Lattesstruc}. 

It is elementary to check that a map $A\in \Aut(\C)$ is
$G$-equivariant (see~\eqref{eq:fequivariant2}) if and only if it
is of the form 
$A(z) = n z + \beta$ with  $n\in \Z\setminus \{0\}$ and
$\beta\in \C$; so the map $A$ in \eqref{eq:inf_inf_diagram}
corresponds to the case $\beta=0$. It is easy to see that for general $\beta\in \C$ the map obtained as a quotient of $A$ on $\CDach_0=\C^*$ as in \eqref{eq:inf_inf_diagram}
is conjugate to $f(z)=z^n$ (for $n\in \Z\setminus\{-1,0,1\})$.

The  
canonical orbifold metric\index{canonical orbifold!metric}\index{metric!canonical orbifold}\index{o@$\omega$}\index{orbifold!canonical metric} 
$\omega$ of
$f(z)=z^n$ (see Section~\ref{sec:expratThmaps} and in particular
the discussion before Proposition~\ref{prop:conforbsym}) is the
conformal metric on $\C^*$ with length element 
$|dz|/(2\pi |z|)$ 
(see Section~\ref{sec:metrspterm} for the terminology).   Equipped  with this metric, $\CDach_0=\C^*\cong \C/G$  is isometric to an  infinite cylinder.

We can describe the maps $\chi=\pm \chi_n$ with $n\in \N\setminus\{1\}$ obtained from Chebyshev polynomials in a similar vein. For these maps we have 
$\alpha_f(-1)=\alpha(1)=2$ and $\alpha_f(\infty)=\infty$. So the orbifold $(\CDach, \alpha_f)$ of $f$ has a puncture at $\infty$ and so $\CDach_0=\C$. The 
universal orbifold covering map $\Theta\: \C \ra \CDach_0=\C$ is given by 
$\Theta(z)=\cos(2\pi z)$ for $z\in \C$. To see this, note that  $\Theta\: \C \ra \C$ is a branched covering map with the critical values $-1$ and $1$ and that 
$\deg(\Theta, z)=2$ whenever $z\in \Theta^{-1}(\{-1,1\})$. 

Clearly, for $z,w\in \C$ we have 
\begin{align}
  \label{eq:cos_ind_Dinfty}
  \cos(2\pi z)&= \cos(2\pi w) 
  \text{ if and only if }\\
  \notag
  w&= \pm z + k \text{ for some $k\in \Z$}. 
\end{align}

So $\Theta$ is induced by the group  $G$ 
of isometries of the form $z\mapsto \pm z + k$ with  $k\in \Z$, i.e., $G=\widetilde G$
with $\widetilde G$  as in  case
$(22\infty)$ of Lemma~\ref{lem:frieze_gps}.  As before, we identify 
$\C/\Gtr$ with $\C^*$ and consider $\pi(z)=\exp(2\pi \iu z)$ for $z\in \C$ as the quotient map $\pi \: \C\ra \C^*$.  Let $\overline \Theta(z)=\frac12(z+1/z)$ for $z\in \C^*$. Then 
$\Theta=\overline \Theta\circ \pi$. 

A map $A\in \Aut(\C)$ is $G$-equivariant 
if and only if it is of the form
$A(z) = n z + l/2$ where $n\in \Z\setminus \{0\}$ and
$l\in \Z$ (we omit the elementary proof for this
fact). For  $n\in \N\setminus\{1\}$ and $l\in 
\{0,1\}$ the map $A$ descends to the map $\chi=\pm\chi_n$ on $\CDach_0=\C$ under the map $\Theta\: \C\ra \C$. Actually, if we define $\overline A(z)=(-1)^l z^n$
for $z\in \C^*$, then we have the commutative diagram 
\begin{equation}
  \label{eq:chin_comm-dia}
  \xymatrix{
    \C \ar[r]^{A} \ar[d]_{\pi} 
    \ar@/_2pc/[dd]_{\Theta(z)= \cos(2\pi z)} 
    & \C \ar[d]^{\pi}\ar@/^2pc/[dd]^{\Theta(z)= \cos(2\pi z)}
    \\
    \C^* \ar[r]^{\overline{A}}\ar[d]_{\overline{\Theta}} 
    & \C^*\ar[d]^{\overline{\Theta}}
    \\
    \C \ar[r]^{\chi} & \C_{\phantom{0}}\!\!\rlap{.}
  }
\end{equation}
This is again analogous to \eqref{eq:holom_qu_diagram} obtained for Latt\`es maps.

Based on the considerations in Section~\ref{sec:expratThmaps} it is not hard to see that 
the canonical orbifold metric $\om$ of $\chi$ is given by the length element
 $$\frac{|dz|}{2\pi |z-1|^{1/2} |z+1|^{1/2}}$$ on $\C$. We will describe a more geometric picture for $(\C, \om)$ 
that will lead to  models for the maps $\pm \chi_n$ similar to
the models for Latt\`{e}s maps as described in
Section~\ref{sec:Lattes} and Section~\ref{sec:examples-lattes-maps}.

Our universal orbifold covering map $z\mapsto \Theta(z)=\cos(2\pi z)$ is induced by the group $G=\widetilde{G}$ as in case $(22\infty)$ of 
Lemma~\ref{lem:frieze_gps}.  So we can identify 
the quotient space $\C/G$ with the target $\C$ of $\Theta\: \C\ra \C$ and consider $\Theta\: \C \ra \C\cong \C/G$ as the quotient map. If we denote by 
$[u]\in \C/G$  the equivalence class of a point $u\in \C$ under the equivalence relation  on $\C$ induced by  $G$, then this identification is more explicitly given 
by the well-defined map $[z]\in \C/G \mapsto \Theta(z)\in \C$.  

The 
canonical orbifold metric\index{canonical orbifold!metric}\index{metric!canonical orbifold}\index{o@$\omega$}\index{orbifold!canonical metric}\index{push-forward!of metric!by orbifold covering map} 
$\om$ is essentially the push-forward of the Euclidean metric under the quotient map.  Under our identification
 $\C\cong \C/G$ we have  
\begin{equation}
  \label{eq:def_d_cos}
  \om([x],[y])\coloneqq  \inf\{\abs{z-w} : z\in [x], w\in [y]\}
\end{equation}
for $[x],[y]\in \C/G$ (see Section~\ref{sec:expratThmaps} and  in particular
\eqref{eq:defcanorbmetr0}).  
%

The metric space $(\C/G, \om)$ is isometric to the space $\Delta$ obtained by gluing two copies of the half-strip $S=[0,1/2]\times [0,\infty)\sub \R^2\cong \C$ together along their boundaries.
Here the half-strips carry the Euclidean metric and $\Delta$ the induced path metric. 

To see this, note that the strip
$F= [0,1/2]\times \R \subset \R^2 \cong \C$ is a fundamental
domain of $G$ (see Section~\ref{sec:appquotmaps}).  Under the action of $G$ two points in the
boundary of $F$ are identified if they are complex conjugates of
each other. So the quotient $\C/G$ is obtained by folding the
strip $F$ along the real axis and gluing together corresponding
boundary parts of the 
half-strips
$S_\wt\coloneqq [0,1/2]\times [0,\infty)$ and
$S_\bt\coloneqq [0,1/2]\times (-\infty,0]$. The metric $\om$ is
a path metric and corresponds to the Euclidean metric on $S_\wt$
and $S_\bt$. So  $(\C/G, \om)$ and $\Delta$ are
indeed isometric. 
  
  In the following we will identify 
  these spaces. We will also consider the half-strips $S_\wt$ and $S_\bt$ as subsets and sides of $\Delta$. We color $S_\wt$  white, and 
   $S_\bt$ black. Then $\Delta$  is a locally Euclidean surface with two conical
singularities as indicated on the right in
Figure~\ref{fig:chebyshev}; the conical singularities 
are labeled by $1$ and $-1$, 
because they 
correspond to these points under the identification $ \C/G\cong\C$. Indeed, 
$ [0]\cong \Theta(0)=1$ 
and $[1/2] \cong \Theta(1/2)=-1$.    

\begin{figure}
  \centering
  \begin{overpic}
    [width=10cm, tics=20,
    ]
    {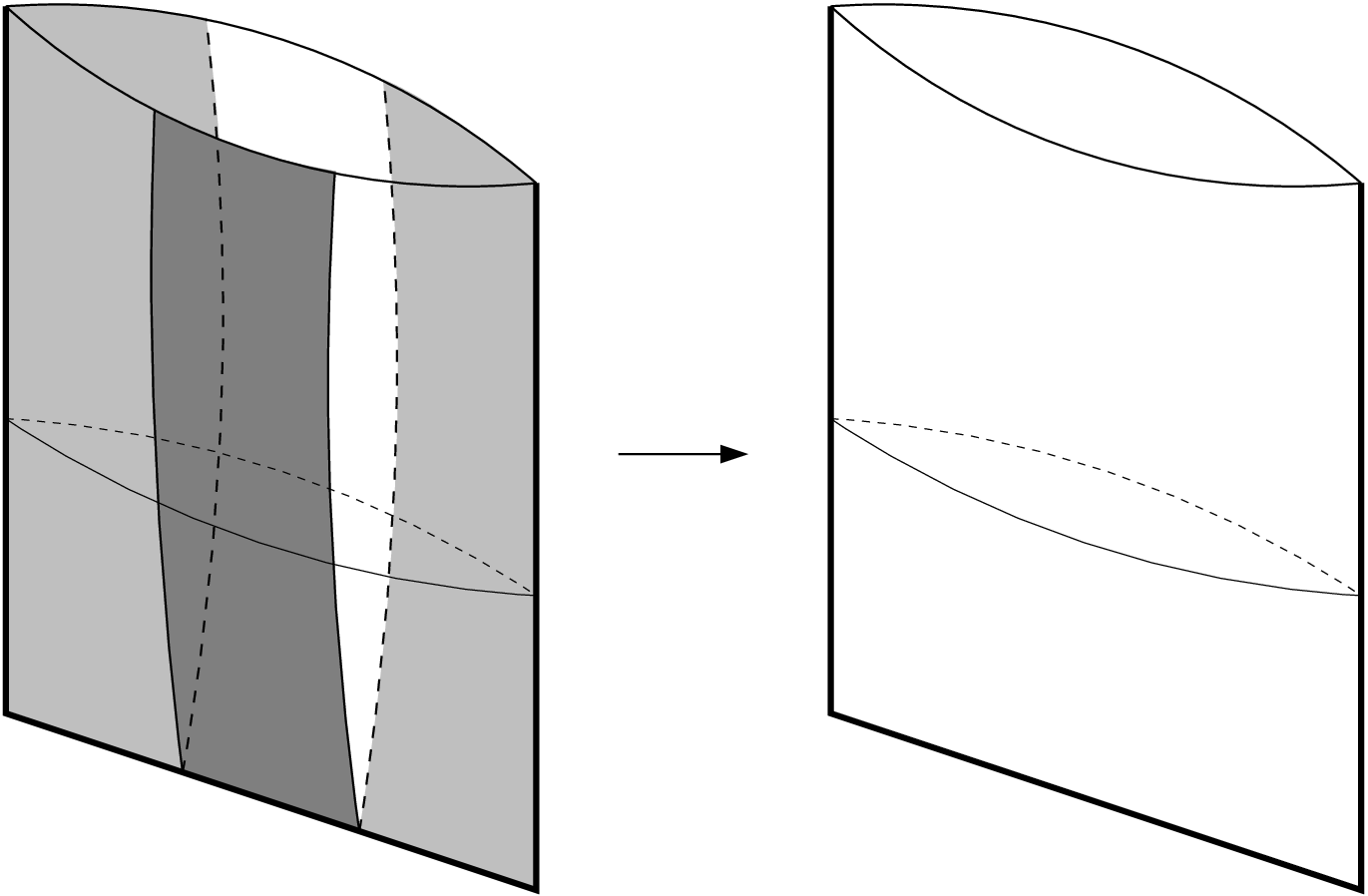}
    \put(48,35){$h$}
    \put(-3,9.5){${\scriptstyle 1\mapsto 1}$}
    \put(11,5){${\scriptstyle \mapsto -1}$}
    \put(24,1.5){${\scriptstyle \mapsto 1}$}
    \put(33.5,-2.5){${\scriptstyle -1\mapsto -1}$}
    \put(60,10){${\scriptstyle 1}$}
    \put(97.5,-2.5){${\scriptstyle -1}$}
    \put(90,60){$\Delta$}
  \end{overpic}
  \caption{Model for a Chebyshev polynomial.}
  \label{fig:chebyshev}
  \index{Chebyshev polynomial}
\end{figure}
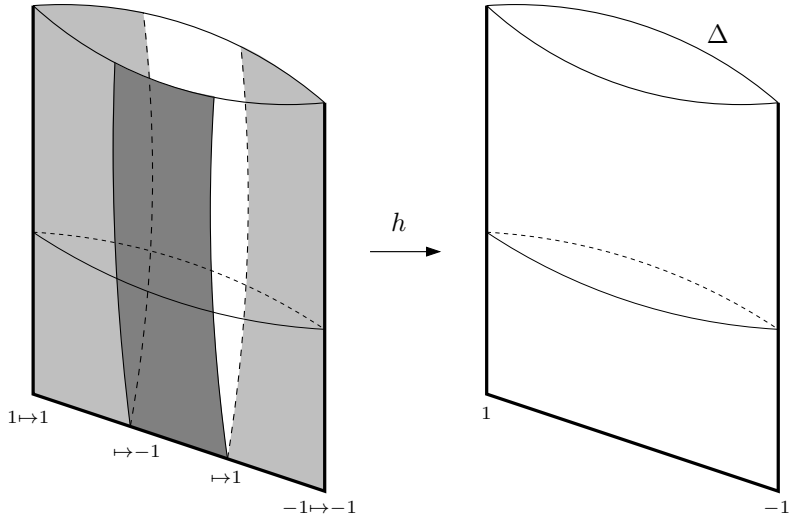

%
%

We now fix  $n\in \N$ and divide  each side  of $\Delta$ into
$n$ (half-)strips $S'$ of equal size. Then  each strip $S'$ is
isometric to $[0, \frac{1}{2n}]\times[0,\infty)$,  and hence similar to $S$ by the  scaling factor
$n$. We color the strips $S'$ in a checkerboard fashion
black and white so that  strips sharing an edge have different colors. 

One can now define a 
map $h\colon \Delta \to \Delta$ as follows. We map
each white strip $S'$ to $S_\wt$ and  each black
strip $S'$  to $S_\bt$ by an orientation-preserving Euclidean similarity. Then  $h$ is
well-defined, because the definitions for $h$ match  on the edges where two strips intersect. An
example of such a map is indicated  in Figure~\ref{fig:chebyshev}.

These maps $h$ give a Euclidean model for the maps $\pm \chi_n$ due to the following fact.

\begin{prop}
  \label{prop:chebychev_model} Let $n\in \N$. 
 Then  every map $h\colon \Delta \to \Delta$ obtained from  the above  construction 
 is topologically conjugate to $\chi_n$
  or $-\chi_n$. Conversely, every polynomial $\chi_n$ and $-\chi_n$ is
  topologically conjugate to such a map $h$. 
\end{prop}

\begin{proof} Fix $n\in \N$. 
  Then the (half-)strips $S'$ of the form 
   $$\textstyle [\frac{k}{2n}, \frac{k+1}{2n}] 
  \times [0, \infty) \sub S_\wt\text{ and }[\frac{k}{2n}, \frac{k+1}{2n}]  \times (-\infty,0]\sub S_\bt$$ 
  for  $k=0, \dots,n-1$ divide the sides $S_\wt$ and $S_\bt$ of $\Delta$, respectively. 
  The checkerboard coloring of these strips $S'$ is uniquely determined if we specify the coloring 
  of  $S'_0\coloneqq [\frac{0}{2n}, \frac{1}{2n}] \times [0, \infty)$. 
  
  If $S_0'$ is colored white, then 
  the map $A(z)=nz$ passes to the quotient $\Delta=\C/G$ and sends white strips $S'$ to $S_\wt$ and black strips $S'$ to $S_\bt$ by a Euclidean similarity. In other words, $A$ induces the map 
  $h\: \Delta\ra \Delta$ 
  discussed above. On the other hand, by \eqref{eq:chin_comm-dia}  this map $A$  passes to the quotient $\chi_n$
  under the map $z\mapsto\Theta(z)=\cos(2\pi z)$. This implies that the induced homeomorphism 
  $\widetilde\Theta: \Delta=\C/G\ra \C$ defined as 
   $\widetilde \Theta([z])=\Theta(z)$ for $[z]\in \C/G$  gives the conjugacy $h= \widetilde\Theta^{-1} \circ \chi_n\circ 
  \widetilde\Theta$.  
  
  If  $S_0'$ is colored black, then $A(z)=nz+1/2$ passes to the quotient $\Delta=\C/G$ and a strip 
    $S'$ is sent to a strip $S_\wt$ or  $S_\bt$ of the same color  by a Euclidean similarity. 
  So again $A$ induces the map $h$ and by \eqref{eq:chin_comm-dia} we get a conjugacy 
  $h= \widetilde\Theta^{-1} \circ \chi\circ 
  \widetilde\Theta$, where $\chi=-\chi_n$.   
\end{proof}

\chapter{Visual Metrics}
\label{cha:visual-metrics}
\index{visual metric}
\index{metric!visual}
\index{r@$\varrho$}

In this chapter we construct a natural class of metrics for an
expanding Thurston map that we call \emph{visual metrics}.  We haven
chosen this name, because there is a close relation between these
metrics and visual metrics on the boundary at infinity of a Gromov
hyperbolic space. Indeed, for an expanding Thurston map $f\: S^2\ra
S^2$ one can define a Gromov hyperbolic \emph{tile graph} whose
boundary at infinity can naturally be identified with $S^2$.  By this
identification, a metric $\varrho$ on $S^2$ is visual in the sense of Gromov
hyperbolic spaces if and only if it is visual as it will be defined in
this chapter (see Chapter~\ref{cha:Gromov} and in particular
Theorem~\ref{thm:visualGrTh}). In general, a visual metric $\varrho$ is not a
length metric on $S^2$.  In Chapter~\ref{cha:geom-visu-sphere} we will investigate the resulting metric space $(S^2, \varrho)$ in more
detail.

We will  first give a quick overview of the
definition and the basic properties of visual metrics. In
Sections~\ref{sec:number-mx-y} and \ref{sec:exist-visu-metr} we will
then provide the technical details. We conclude this chapter with  
Section~\ref{sec:orbifold_visual} where we consider rational expanding
Thurston maps. In particular,  we show that the canonical orbifold
metric (see Section~\ref{sec:expratThmaps}) 
for such a map $f$
is a visual metric precisely if $f$ is a Latt\`{e}s map 
(see Proposition~\ref{prop:orbivispara}). 

Let $f\:S^2\ra S^2$ be an expanding Thurston map, and $\CC\subset S^2$ be
a Jordan curve with $\post(f)\subset \CC$. We consider the cell decompositions of $S^2$ 
for  $(f,\CC)$ as defined  in Section~\ref{sec:tiles}. One can 
think of the set of $n$-tiles as a discrete approximation of the
sphere $S^2$ and measure distances of points by a quantity related to combinatorics of $n$-tiles.

Indeed, let  $x,y\in S^2$ be   two distinct points, and  $X$ and  $Y$ be $n$-tiles
 with $x\in X$, $y\in Y$. Since $f$ is expanding,   $X$ and $Y$ must be disjoint if   $n$ is 
 sufficiently large (see
\eqref{eq:defexpXn} and Lemma~\ref{lem:exp_ind_C}). 
This leads to the following definition. 

\ifthenelse{\boolean{nofigures}}{}{
\begin{figure}
  \centering
  \begin{subfigure}{0.48\textwidth}
    \begin{overpic}
      [width=2.395in, 
      tics=20]{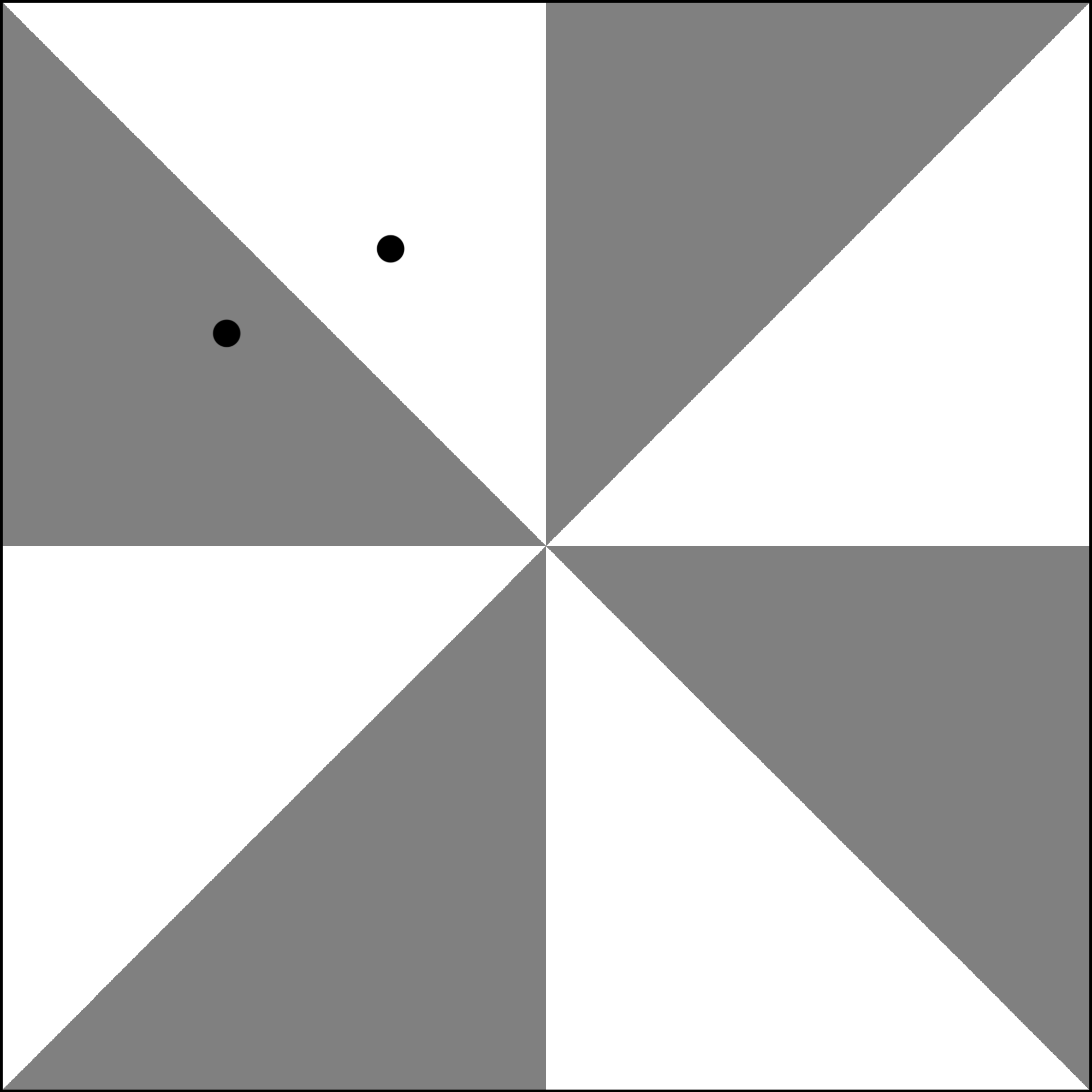}
      \put(15,68){$x$}
      \put(30,76){$y$}
    \end{overpic}
  \end{subfigure}
  \hspace*{\fill}
  \begin{subfigure}{0.48\textwidth}
    \begin{overpic}
      [width=2.395in, 
      tics=20]{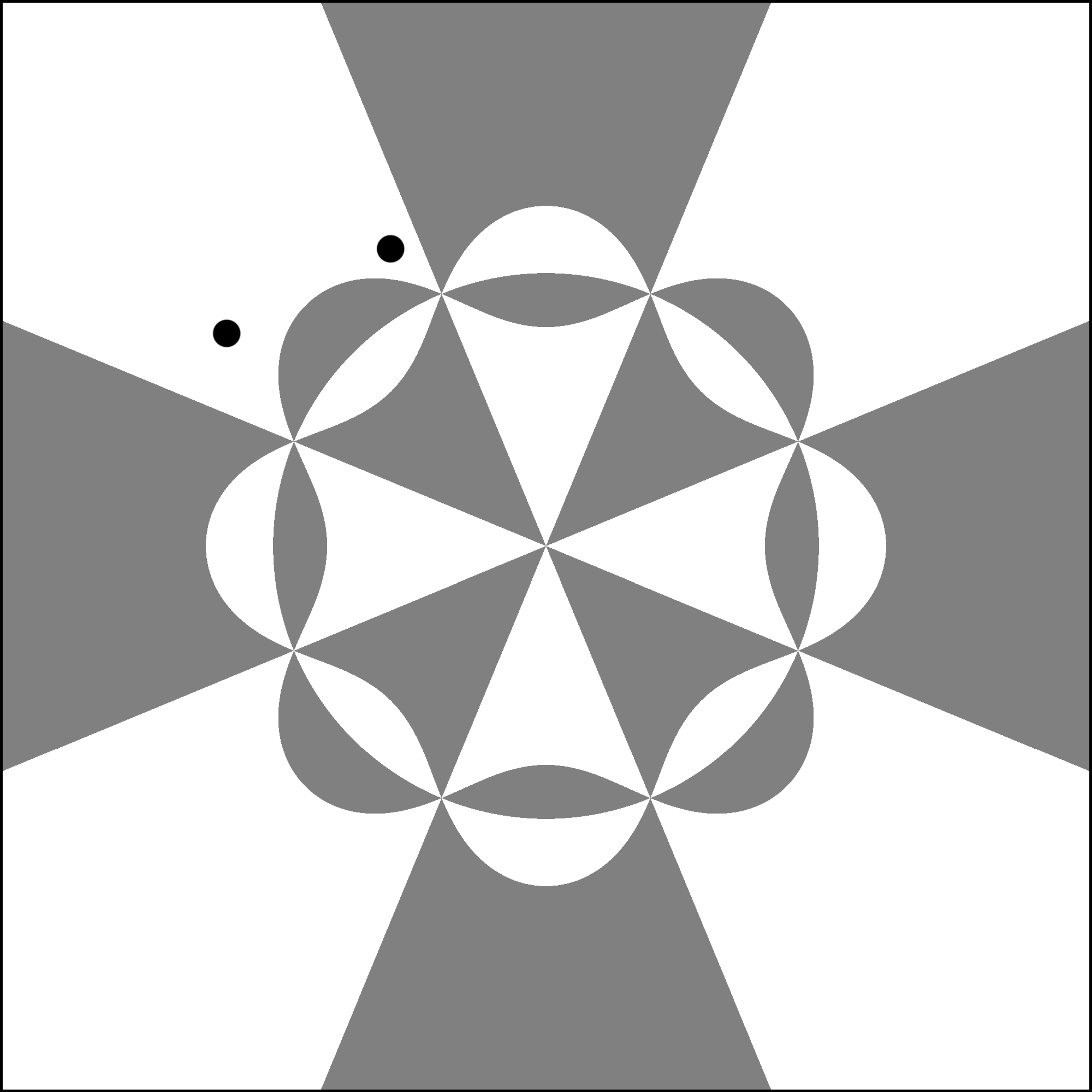}
    \end{overpic}
  \end{subfigure}

  \vspace{0.04\textwidth}
  \begin{subfigure}{0.48\textwidth}
    \begin{overpic}
      [width=2.395in, 
      tics=20]{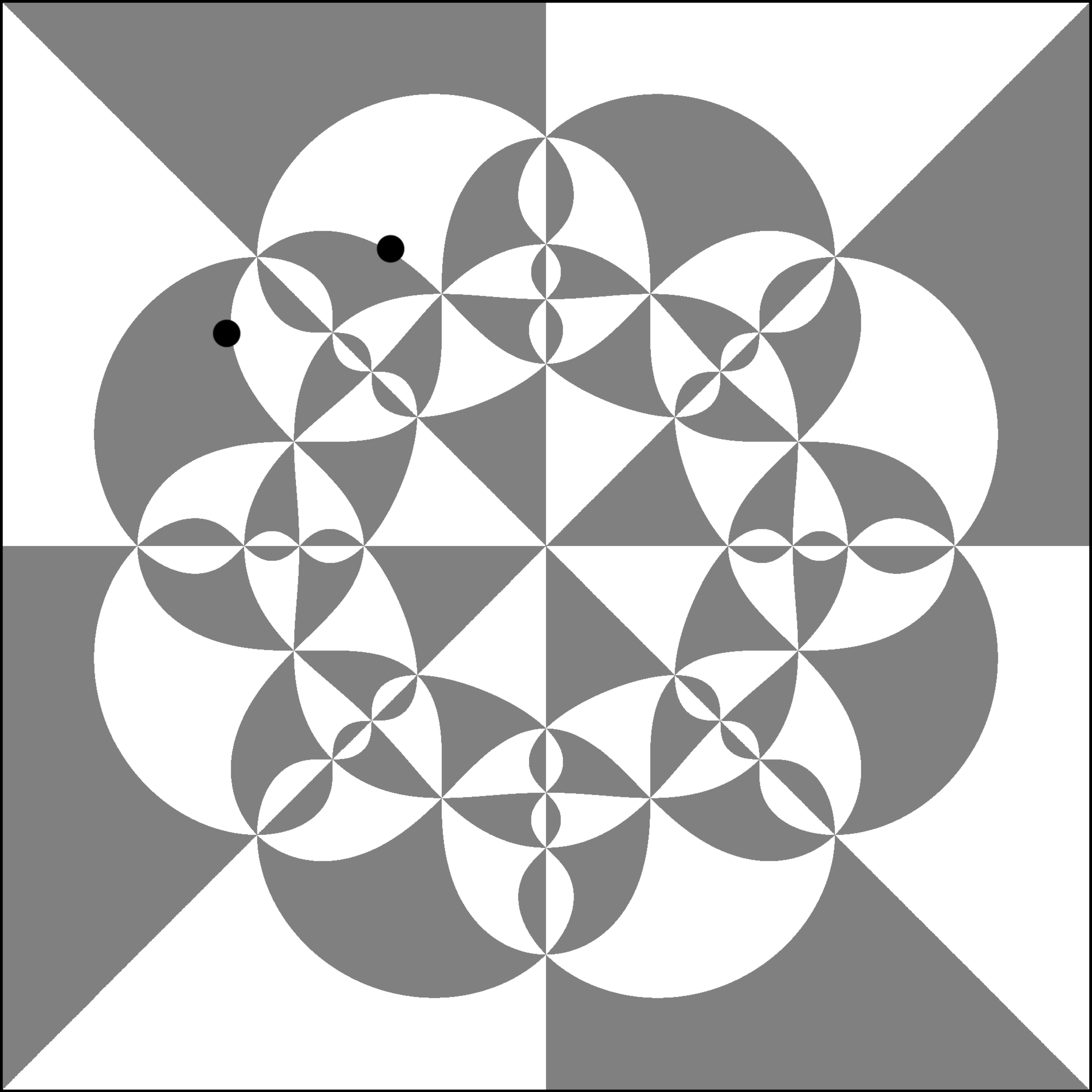}
    \end{overpic}
  \end{subfigure}
  \hspace*{\fill}
  \begin{subfigure}{0.48\textwidth}
    \begin{overpic}
      [width=2.395in, tics=20]{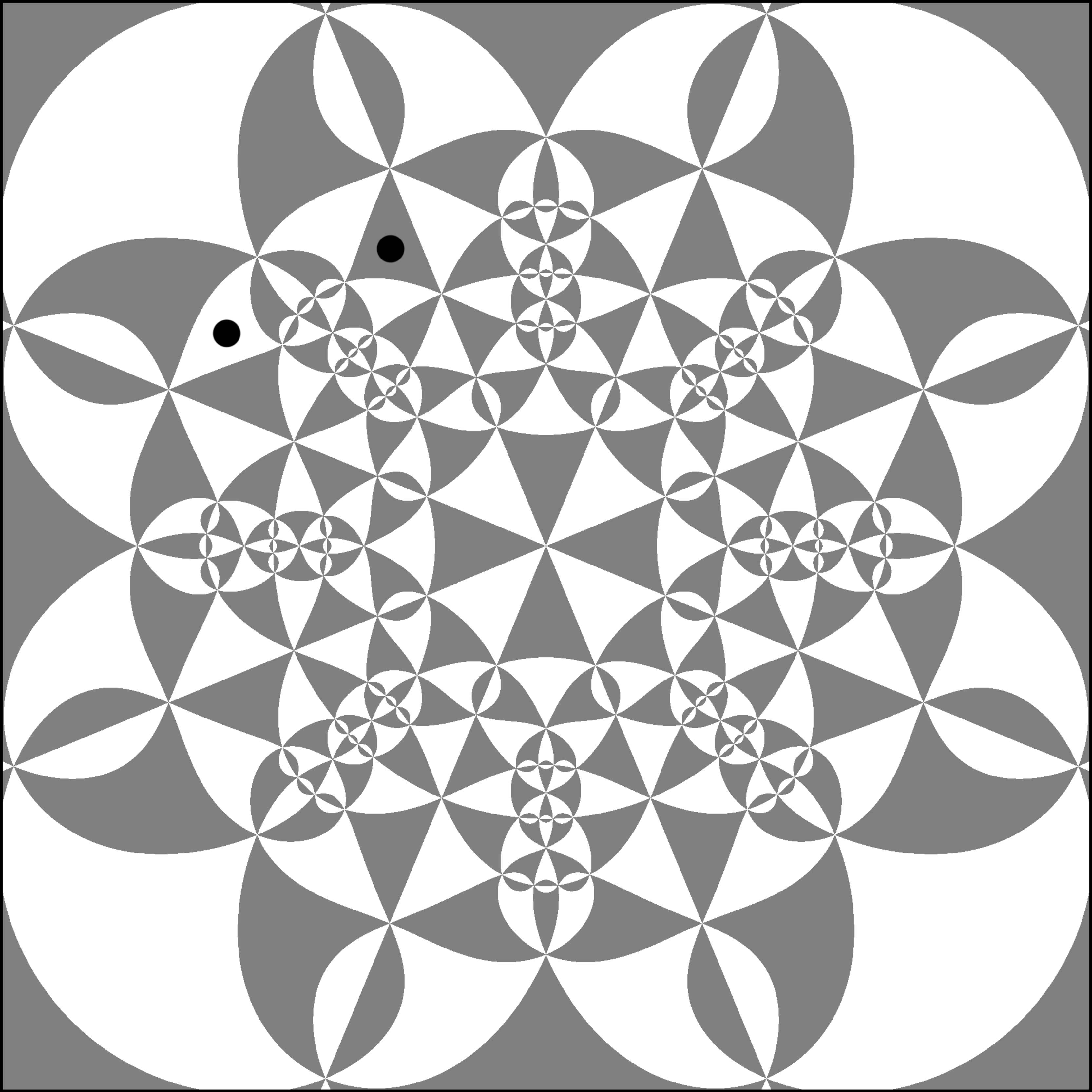}
    \end{overpic}
  \end{subfigure}
    
  \vspace{0.04\textwidth}
  \begin{subfigure}{0.48\textwidth}
    \begin{overpic}
      [width=2.395in, tics=20]{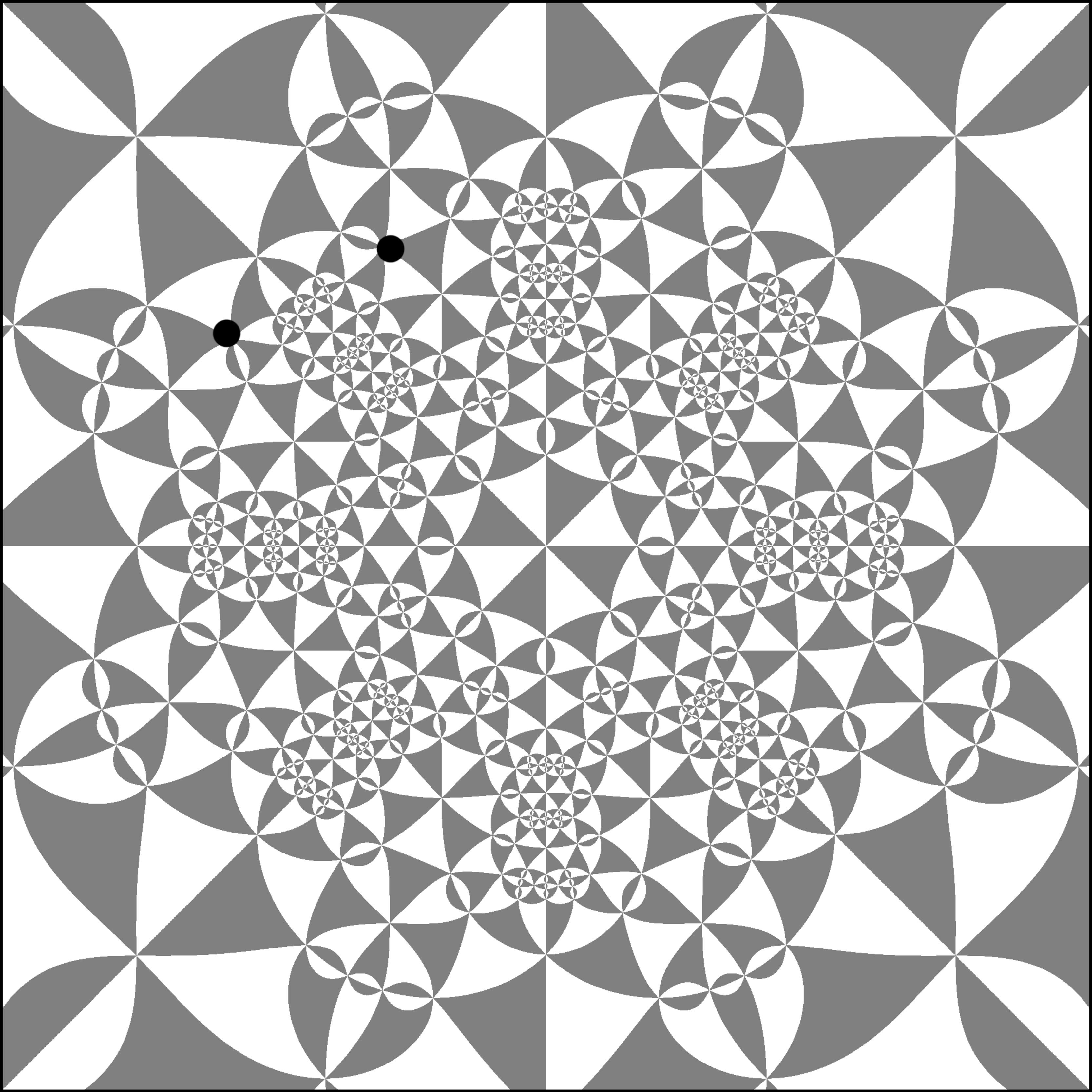}
    \end{overpic}
  \end{subfigure}
  \hspace*{\fill}
  \begin{subfigure}{0.48\textwidth}
    \begin{overpic}
      [width=2.395in, tics=20]{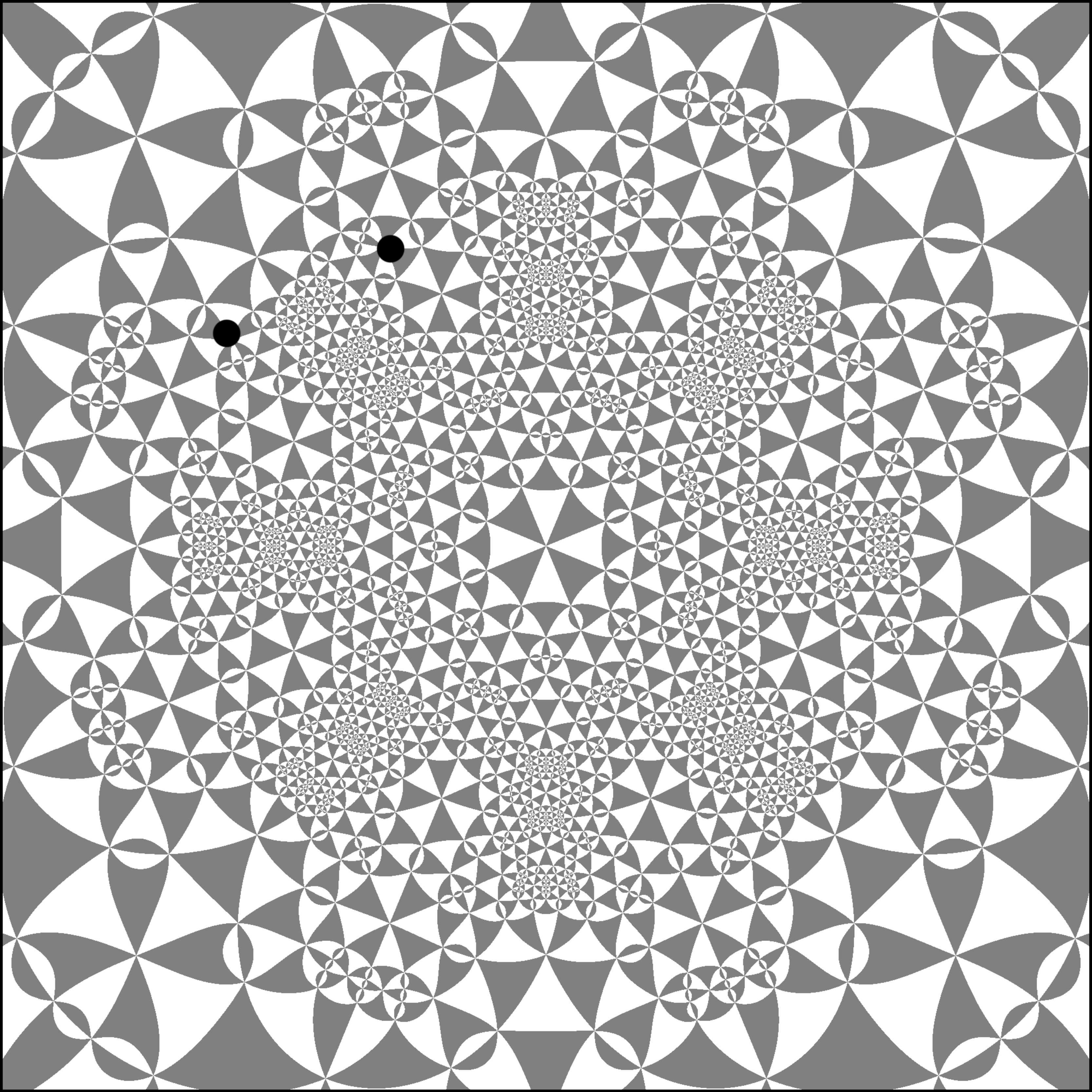}
    \end{overpic}
  \end{subfigure}
  \caption{Separating points by tiles.}
  \label{fig:def_visual_metric}
\end{figure}

}

\begin{definition}
  \label{def:mxy} 
  Let $f\: S^2\ra S^2$ be an expanding
  Thurston map,  $\CC\sub S^2$ be a Jordan curve with
  $\post(f)\sub \CC$, and $x,y\in S^2$.   
  For $x\ne y$ we define\index{m@$m_{f,\CC}$|textbf} 
 \begin{align*} 
   m_{f,\CC}(x,y)
   \coloneqq 
   \max\{n\in \N_0: {}&\text{there exist non-disjoint  $n$-tiles }
   \\ 
   &\text{$X$ and $Y$  for $(f,\CC)$ with $x\in X$, $y \in
     Y$}\}. 
 \end{align*} 
 If $x=y$  we define $m_{f,\CC}(x,x)\coloneqq \infty$. 
\end{definition}

Note that $m_{f,\CC}(x,y)\in \N_0$ 
if $x\ne y$. 
We usually drop both subscripts  in 
$m_{f,\CC}(x,y)$ if $f$ and $\CC$ are clear from the context. A
similar combinatorial quantity that is essentially  equivalent to
$m_{f,\CC}(x,y)$ (see Lemma~\ref{lem:mprops}~\ref{item:mprops5}) is  
\begin{align} 
  \label{eq:def_mprime}
  \index{m`@$m'_{f,\CC}$}
  m'_{f,\CC}(x,y)\coloneqq \min\{n\in \N_0: {}&\text{there exist disjoint
    $n$-tiles } 
  \\ \notag
  &\text{$X$ and $Y$  for $(f,\CC)$ with $x\in X$, $y \in Y$}\}
 \end{align} 
for $x\neq y$, and $m'_{f,\CC}(x,x)\coloneqq  \infty$.


These quantities are illustrated in  
Figure~\ref{fig:def_visual_metric}. Here we 
use 
the map
$f\: \CDach\ra \CDach$ given by 
\begin{equation*}
  f(z)= \iu \frac{z^4-\iu}{z^4+\iu}
\end{equation*}
for $z\in \CDach$ (we will consider this map again in
Example~\ref{ex:noinvCC}).  We have $\post(f)= \{1,\iu, -\iu\}$,
and so we can choose the unit circle $\CC=\partial \D$ as a
Jordan curve containing the postcritical set of $f$. In
Figure~\ref{fig:def_visual_metric} the $n$-tiles for $(f,\CC)$
are shown for $n=1, \dots, 6$. For the points $x$ and $y$ as
indicated in the figure, we have $m_{f,\CC}(x,y)= 3$ and
$m'_{f,\CC}(x,y)= 4$.

The number $m_{f,\CC}(x,y)$ is large if $x$ and $y$ are close together, i.e.,
if $n$-tiles of high level $n$ are needed to separate the points. 
This is the basis of the following definition. 

\begin{definition}[Visual metrics]\label{def:visual} 
  Let $f\: S^2\ra
  S^2$ be an expanding Thurston map. A metric $\varrho$ on $S^2$
  is called a 
{\em visual metric}\index{visual metric|textbf}\index{metric!visual|textbf}\index{r@$\varrho$|textbf}
  (for  $f$) if there exists a
  Jordan curve $\CC\sub S^2$ with $\post(f)\sub \CC$, and a constant
  $\Lambda>1$ such that 
\begin{equation}\label{visual} 
  \varrho(x,y)\asymp \Lambda^{-m(x,y)}
\end{equation}
for all $x,y\in S^2$, where $m(x,y)=m_{f,\CC}(x,y)$ and where the constant
$C(\asymp)$ is independent of $x$ and $y$. 
\end{definition}

Here we use the convention $\Lambda^{-\infty}=0$. The number $\Lambda$
is called the 
{\em expansion factor}\index{expansion factor|textbf}\index{L@$\Lambda$|textbf}
of the metric $\varrho$.  It is easy to
see that the expansion factor of each visual metric is uniquely
determined. Different visual metrics may have different expansion
factors. 

As mentioned above, it is possible to identify the sphere $S^2$ with the boundary at
infinity of a certain Gromov hyperbolic graph constructed from tiles. 
Under this identification, the numbers $m_{f,\CC}(x,y)$ and
$m'_{f,\CC}(x,y)$ are the Gromov product of $x$ and $y$, up to some
additive constants (see Section~\ref{sec:Grhyp},
Chapter~\ref{cha:Gromov}, and Lemma~\ref{lem:m_Gromov}).  

Obvious questions are  whether  visual metrics exist, and 
 how they depend on the chosen Jordan curve $\CC$ and 
the expansion factor $\Lambda$. This is answered by the following proposition.  


\begin{prop}
  \label{prop:visualsummary}
  \index{visual metric}
  \index{metric!visual}
 For an expanding Thurston map $f\:S^2\ra S^2$ the following
 statements are true: 

\begin{enumerate}
\item
  \label{item:exvisual}
  There exist  visual metrics for $f$.

\item 
  \label{item:vistop}
  Every visual metric induces the given   topology on $S^2$.

\item
  \label{item:visCC}
  Let $\varrho$ be  a visual metric for $f$ with expansion factor
  $\Lambda$,  $\widetilde \CC \subset S^2$ be a Jordan curve with
  $\post(f)\subset \widetilde \CC$, and $m=m_{f,\widetilde \CC}$ be
  defined as in Definition~\ref{def:mxy}. 
  Then a relation 
  as in \eqref{visual} is true with the same
  expansion factor $\Lambda$, where the constant 
  $C(\asymp)$ depends on
  $\widetilde\CC$. 
\item
  \label{item:visd1d2}
  Any two visual metrics are 
  snowflake equivalent,\index{snowflake equivalent} 
  and bi-Lip\-schitz
  equivalent if they 
  have the same expansion factor $\Lambda$.  
\item
  \label{item:vsfF}
  \index{Thurston map!iterate of}
  \index{iterate of Thurston map}
  \index{F f@$F=f^n$}
  A metric $\varrho$ is a visual metric for some iterate $F=f^n$ with $n\in \N$ if
  and only if 
  it is a visual metric for $f$. If $\Lambda>1$ is the expansion factor of $\varrho $ for $f$, then $\Lambda_F=\Lambda^n$  is the
  expansion factor of $\varrho$ for $F=f^n$. 
  
  \item
  \label{item:vsLip}
 If $\varrho$ is a visual metric for $f$, then  $f\: (S^2, \varrho)\ra  (S^2, \varrho)$ is a Lipschitz map.

\end{enumerate}

\end{prop}

The notions of snowflake and bi-Lipschitz equivalence were defined in
Section~\ref{sec:QCgeom}.  The proof of
Proposition~\ref{prop:visualsummary} will be provided  in 
Section~\ref{sec:exist-visu-metr}. Theorem~\ref{thm:visexpfactors1} gives a stronger result on the existence
of visual metrics.

In Section~\ref{sec:int-frac-sph} and Section~\ref{sec:snowballs} we introduced 
visual metrics on a more intuitive level, where we considered 
certain self-similar fractal spheres 
constructed as  limits of polyhedral surfaces $\mathcal{S}^n$. Each
surface $\mathcal{S}^n$ was built from tiles   whose  size was about  
$\Lambda^{-n}$ for some constant $\Lambda>1$. A similar statement
is true  in general  for  visual metrics and gives in fact a 
characterization of these metrics.  

\begin{prop}[Characterization of visual metric]
  \label{lem:expoexp} 
  Let $f\: S^2\ra S^2$ be an expanding Thurston map and $\varrho$
  be a metric on $S^2$. 
   Then $\varrho$ is a visual metric for
  $f$ with expansion factor $\Lambda>1$ if and only if the
  following two 
conditions hold for all $n\in \N_0$:
 
  \begin{enumerate} 
  
  \item
    \label{item:expoex1}
    $\dist_\varrho(\sigma, \tau)  \gtrsim \Lambda^{-n}$,  whenever
    $\sigma$ and $\tau$ are disjoint $n$-cells.  

    \item
    \label{item:expoex2}
      $\diam_\varrho(\tau) \asymp \Lambda^{-n}$ 
      for all $n$-edges and all $n$-tiles $\tau$.   
  \end{enumerate}
   Here cells are defined in terms of some Jordan curve $\CC\sub S^2$
  with $\post(f)\sub \CC$, and 
the constants 
$C(\gtrsim)$ and
  $C(\asymp)$   are 
independent of the cells and their level $n$.
\end{prop}
 We will prove
this proposition in Section~\ref{sec:exist-visu-metr}. 

 
For a rational expanding Thurston map $f\colon \CDach \to \CDach$
there are some other natural metrics on the Riemann sphere
$\CDach$ besides the visual metrics for $f$, in particular the
chordal metric $\sigma$ and the canonical orbifold metric
$\omega$ of $f$ as introduced in
Section~\ref{sec:orbif-assoc-thurst}. The chordal metric
$\sigma$ is never visual for $f$ (see
Lemma~\ref{lem:chord_not_vis}).  For the canonical orbifold
metric\index{canonical orbifold!metric}\index{metric!canonical orbifold}\index{o@$\omega$}\index{orbifold!canonical metric} we
will prove the following statement in
Section~\ref{sec:orbifold_visual}.
\begin{prop}[Canonical orbifold metric as visual metric] 
  \label{prop:orbivispara} 
  \index{visual metric}\index{metric!visual}\index{Latt\`{e}s map} Let $f\: \CDach\ra \CDach$ be a rational Thurston map
  without periodic critical points, and $\om$ be the canonical
  orbifold metric for $f$. Then $\om$ is a visual metric for $f$
  if and only if $f$ is a Latt\`{e}s map.
\end{prop}

\section{The number $m(x,y)$}
\label{sec:number-mx-y}

We now turn to a more detailed exposition of 
the  basic properties of the quantity $m(x,y)=m_{f,\CC}(x,y)$
as in Definition~\ref{def:mxy}. 
In the following, $f\: S^2\ra S^2$ will be an expanding Thurston map and $\CC\sub S^2$ be a Jordan curve with $\post(f)\sub \CC$. Note 
that  $\#\post(f)\ge 3$, because $f$ is expanding (see Lemma~\ref{lem:no<3}).

Recall the quantity $D_n=D_n(f,\CC)$ as  defined in \eqref{def:dk} that  measures distances  in terms of
lengths of tile chains. We consider a
slight variant here.

 We define 
$\widetilde D_n=\widetilde D_n(f,\CC)$ as the minimal number of tiles
of levels $k\geq n$ for $(f,\CC)$  required to join opposite sides of
$\CC$, i.e., the smallest number $N\in \N$ for which there are tiles
$X_i\in\bigcup_{k\geq n} \X^k$, $i=1,\dots, N$,  such that
$K=\bigcup_{i=1}^N X_i$ is connected and joins  opposite sides of
$\CC$ (see Definition~\ref{def:connectop}). 
   
While the sets $K$ used to define $D_n$ 
are unions of tiles of level $n$,  the sets $K$ in the definition of
$\widetilde D_n$ are unions of tiles of levels $k\ge n$; in particular,
$D_k\ge \widetilde D_n$ for $k\ge n$.

\begin{lemma} 
  \label{lem:Dtoinfty} 
  \index{d0 n@$D_n$} 
  Let $f\: S^2\ra S^2$ be an expanding Thurston map, and
  $\CC\sub S^2$ be a Jordan curve with $\post(f)\sub \CC$. Let
  $D_n=D_n(f,\CC)$ and $\widetilde D_n=\widetilde D_n(f,\CC)$ for
  $n\in \N_0$.

  Then $D_n\to \infty$ and $\widetilde D_n\to \infty$ as
  $n\to \infty$.
\end{lemma}

\begin{proof}
We know that  $D_k\ge \widetilde D_n$ whenever $k\ge n$. So it suffices to show $\widetilde D_n\to \infty$ as $n\to \infty$. 

Let $\delta_0>0$ be defined as in \eqref{defdelta} (for some base metric on $S^2$) and  
suppose  $K=X_1\cup \dots \cup X_N$ is a connected union of tiles of levels $\ge n$ that joins opposite sides of $\CC$.  Then 
\begin{align*}
  \delta_0 &\le  \diam (K)\,\le \, \sum_{i=1}^N\diam(X_i)\\
           &\le  N \max_{i=1, \dots, N} \diam (X_i) \\
           &\le  N \sup_{k\ge n} \mesh(f,k,\CC).
\end{align*}
Putting $c_n\coloneqq \sup_{k\ge n} \mesh(f,k,\CC)$, we conclude that $N\ge \delta_0/c_n$, and so $\widetilde D_n\ge  \delta_0/c_n$. 

Since $f$ is expanding we have $\mesh(f,n,\CC)\to 0$ and so also $c_n\to 0$ as $n\to \infty$. This implies that $\widetilde D_n
\to \infty$ as desired.
\end{proof}

If $f$ is expanding and $\CC$ is given, then in  view of the last lemma, we can find a number $k_0=k_0(f,\CC)\in \N$ such that 
\begin{equation}\label{def:k0}
\widetilde D_{k_0}=\widetilde D_{k_0}(f,\CC)\ge 10.
\end{equation} 
This inequality will be useful in the following.

 In the next lemma we collect some of the properties of the function $m_{f,\CC}$.

\begin{lemma}
  \label{lem:mprops} 
  \index{m@$m_{f,\CC}$}
  Let $f\: S^2\ra S^2$ be an expanding
  Thurston map,  $\CC\sub S^2$ be a Jordan curve with  
  $\post (f)\sub \CC$, and $m=m_{f,\CC}$. Then the following statements are true:

\begin{enumerate}

\item
  \label{item:mprops1} 
  There exists a number $k_1>0$ such that 
\begin{equation*}
 \min\{m(x,z), m(y,z)\}\le m(x,y)+k_1
 \end{equation*}
  for all 
$x,y,z\in S^2.$  

  \item
    \label{item:mprops2} 
    We have 
      $$m(f(x),f(y))\geq m(x,y)-1$$ for all $x,y\in S^2$.
   
  \item
    \label{item:mprops3} 
    Let $\widetilde \CC\sub S^2$ be another Jordan curve with $\post(f)\sub \widetilde \CC$. 
   Then there exists a constant $k_2>0$ such  that
    \begin{equation*}
      m(x,y)-k_2\leq m_{f,\widetilde\CC}(x,y)\leq m(x,y)+k_2
    \end{equation*}
    for all $x,y\in S^2$.

  \item
    \label{item:mprops4} 
    Let $F=f^n$ for $n\in \N$ be an iterate of $f$. Then there
    exists a constant $k_3>0$ such that     
    \begin{equation*}
      m(x,y)-k_3 \leq n\cdot m_{F,\CC}(x,y) \leq 
      m(x,y)
    \end{equation*}
    for all $x,y\in S^2$.
    
  \item
    \label{item:mprops5} 
    The quantities $m$ and  $m'=m'_{f,\CC}$ as defined in
    \eqref{eq:def_mprime} 
     are comparable in the following sense:  there exists a constant
    $k_4>0$ such that 
    \begin{equation*}
      m(x,y)-k_4 \leq m'_{f,\CC}(x,y) \leq m(x,y)+1 
    \end{equation*}
    for all $x,y\in S^2$. 
   \end{enumerate}
 \end{lemma}

In Chapter~\ref{cha:Gromov} we will prove that
$m=m_{f,\CC}$ can   essentially be interpreted as a Gromov product in a suitable
metric space (see Lem\-ma~\ref{lem:m_Gromov}). Property~\ref{item:mprops1} is then related to 
 the 
Gromov hyperbolicity of this space  (compare with \eqref{def:Grprod}).

\begin{proof} We fix $k_0=k_0(f,\CC)\in \N$ as in \eqref{def:k0}.
Let  $x,y\in S^2$  be arbitrary. In order to establish the desired inequalities we may always assume $x\ne y$. Unless otherwise stated, tiles will be for $(f,\CC)$. 

\smallskip
\ref{item:mprops1}  
 Let $m\coloneqq m(x,y)\in
  \N_0$ be as in Definition~\ref{def:mxy}.  We can pick
  $(m+1)$-tiles   $X_0$
  and $Y_0$ containing $x$ and $y$, respectively. Then
 $    X_0\cap Y_0=\emptyset$   by definition of $m$.

 Define $n\coloneqq  m+k_0$, and let $z\in S^2$ be an arbitrary
 point. We
 claim that
    $m(x,z)\le  n \text{ or } m(y,z)\le  n$.

 Otherwise,  $m(x,z)\ge  n+1$ and $m(y,z)\ge n+1$, and so  
  by Definition~\ref{def:mxy} there exist numbers $m_1,m_2\ge  n+1$ and
  $m_1$-tiles $X$ and $Z$ with $x\in X$, $z\in Z$ and $X\cap Z\ne \emptyset$, and  $m_2$-tiles $Y$ and $Z'$ with $y\in Y$, $z\in Z'$ and $X\cap Z'\ne \emptyset$. 
  
  Then the  set $K=X\cup Z\cup Z'\cup Y$ is connected and meets the disjoint 
  $(m+1)$-tiles $X_0$ and $Y_0$. Thus $f^{m+1}(K)$
  joins opposite sides of $\CC$ by Lemma~\ref{lem:maptotop}, and
  consists of  
four  tiles of levels
   $ \ge n-m=  k_0$. This contradicts \eqref{def:k0}, proving the claim.  

  So we have  $m(x,z)\leq m+k_0$ or $m(y,z)\leq m+k_0$. This  implies
  \ref{item:mprops1} with the constant $k_1=k_0$ which is independent
  of $x$ and $y$.

  \smallskip 
  \ref{item:mprops2}
  We may assume that 
  $m\coloneqq m(x,y)\ge 1$. There are non-disjoint
  $m$-tiles $X$ and $Y$ with $x\in X$ and $y\in Y$. 
  It follows that  
  $f(X)$ and $f(Y)$ are non-disjoint $(m-1)$-tiles with $f(x)\in f(X)$ and $f(y)\in f(Y)$. Hence  $m(f(x),f(y))\geq m-1$ as desired.  

  \smallskip 
  \ref{item:mprops3} Let $\widetilde m\coloneqq m_{f,\widetilde
    C}(x,y)\in \N_0$. Then there exist $\widetilde m$-tiles
  $\widetilde X$ and $\widetilde Y$ for $(f,\widetilde \CC)$ with
  $x\in \widetilde X$, $y\in \widetilde Y$, and $\widetilde X\cap
  \widetilde Y\ne \emptyset$.  By Lemma~\ref{lem:tileflower} the sets
  $\widetilde X$ and $\widetilde Y$ are each contained in $M$
  $\widetilde m$-flowers for $(f,\CC)$, where $M$ is independent of $
  \widetilde X$ and $\widetilde Y$. In particular, this implies that
  we can find a chain of at most $2M$ such $\widetilde m$-flowers
  joining $x$ and $y$ (recall that chains  were introduced in  Definition~\ref{def:chains}). 
Since any two tiles in the closure $\overline
  {W^n(v)}$ of an $n$-flower have the point $v$ in common, it follows
  that there exists a chain $X_1, \dots, X_N$ of $\widetilde m$-tiles
  for $(f,\CC)$ joining $x$ and $y$ with $N\le 4M$. Let $x_1\coloneqq x$,
  $x_N\coloneqq y$, and for $i=2, \dots, N-1$, pick a point $x_i\in X_i$.
  Then $m(x_i, x_{i+1}) \ge \widetilde m$ for $i=1, \dots, N-1$. Hence
  by repeated application of \ref{item:mprops1} we obtain
  \begin{align*}
    \widetilde m &\le \min\{m(x_i, x_{i+1}): i=1, \dots, N-1\}
    \\
    &\le 
    m(x_1,x_N)+Nk_1\,\le \,  m(x,y)+4Mk_1. 
  \end{align*}
    Since $4Mk_1$ is independent of $x$ and $y$, we get an upper
    bound as in \ref{item:mprops3}. A lower bound is obtained by
    the same argument if we reverse the roles of $\CC$ and
    $\widetilde \CC$.

    \smallskip \ref{item:mprops4} The map $F$ is also an
    expanding Thurston map, and we have $\post(f)=\post(F)$ (see
    Lemma~\ref{lem:Thiterates}); so the Jordan curve $\CC$
    contains the set of postcritical points of $F$ and
    $m_{F,\CC}$ is defined. 
 The $k$-tiles for $(F,\CC)$ are precisely the
    $(nk)$-tiles for $(f,\CC)$ (see 
    Proposition~\ref{prop:celldecomp}~\ref{item:celdecompiter}). In the ensuing proof we will only
    consider tiles for $(f, \CC)$.

  Let $m_F\coloneqq m_{F,\CC}(x,y)$ and $m \coloneqq m(x,y)$; then  there are non-disjoint $(nm_F)$-tiles $X$ and $Y$ with $x\in X$ and $y\in Y.$  So $m\ge nm_F$ which gives the desired upper bound. 
 
 We claim that on the other hand, we have 
 $m\le nm_F+k_3$, where $k_3=n+k_0-1$.  
 To see this, assume that $$m\ge nm_F+k_3+1=n(m_F+1)+k_0.$$ 
Then  we can find  non-disjoint $m$-tiles $X$ and $Y$ with $x\in X$, $y\in Y$. Moreover, we can pick   $n(m_F+1)$-tiles $X'$ and $Y'$ with $x\in X'$ and $y\in Y'$. By definition of $m_F$ we know that $X'\cap Y'=\emptyset$; so 
  $X'$ and $Y'$ are disjoint $n(m_F+1)$-tiles joined by the connected set $K=X\cup Y$. Hence by Lemma~\ref{lem:flowerbds} the set $K$ must consist of at least $$D_{m-n(m_F+1)}\ge \widetilde D_{k_0}\ge 10$$
  $m$-tiles; but $K$ consists of only two  $m$-tiles. This is a contradiction showing the desired claim. 
  
 \smallskip 
 \ref{item:mprops5}
 Let $m'\coloneqq m'_{f,\CC}(x,y)$ be defined as in
 \eqref{eq:def_mprime}. Then $m'\ge 1$, because the two $0$-tiles have
 non-empty intersection.  So  $m'-1\ge 0$, and there exist
 $(m'-1)$-tiles $X$ and $Y$ with $x\in X$ and $y\in Y$. Then $X\cap
 Y\ne \emptyset$ by definition of  
 $m'$, and so $m(x,y)\ge m'-1$. 
 
 Conversely, let $m \coloneqq m(x,y)$. Suppose $m'< m-k_0$. 
 Then there exist $m'$-tiles $X'$ and $Y'$ with $X'\cap Y'=\emptyset$, $m$-tiles 
 $X$ and $Y$ with $X\cap Y\ne \emptyset$, and $x\in X\cap X'$, $y\in Y\cap Y'$. 
 Hence   $K=X\cup Y$ is a union of two $m$-tiles joining  the disjoint $m'$-tiles   
 $X'$ and $Y'$; but such a union must consist of at least   
 $$D_{m-m'}\ge \widetilde D_{k_0}\ge 10$$ 
$m$-tiles by Lemma~\ref{lem:flowerbds}. This is a contradiction showing that $m-k_0\le m'$. So the claim is true with $k_4=k_0$.  
 \end{proof}

\section{Existence and basic properties of visual metrics}
\label{sec:exist-visu-metr}

We are now ready to prove 
Propositions~\ref{prop:visualsummary}  and~\ref{lem:expoexp}, and in particular  the existence of visual metrics. In this section we will also collect various other and somewhat more technical statements 
related to visual metrics that will be useful later on. 
 
\begin{proof}[Proof of Proposition~\ref{prop:visualsummary}] 
  \ref{item:exvisual} Fix a Jordan curve $\CC\sub S^2$ with
  $\post(f)\sub \CC$.  
  
  A function $q\colon S^2\times
  S^2\to [0,\infty)$ is called a 
  \defn{quasimetric}\index{quasimetric} 
  if it has the symmetry property $q(x,y)=q(y,x)$,
  satisfies the condition $q(x,y)=0\Leftrightarrow x=y$, and the
  inequality
  \begin{equation}
    \label{eq:qmetric}
    q(x,y)\leq K(q(x,z)+q(z,y)),
  \end{equation}
 holds  for a constant $K\geq 1$ and all $x,y,z\in S^2$.
  
 We now  define a quasimetric $q$ on $S^2$. For  this purpose, we fix     $\Lambda>1$
  and  set
  \begin{equation}
    \label{eq:defq}
    q(x,y)\coloneqq \Lambda^{-m(x,y)},
  \end{equation}
 for $x,y\in S^2$,  where $m(x,y)=m_{f,\CC}(x,y)\in \N_0\cup \{\infty\}$  is as in
   Definition~\ref{def:mxy}.
   
   Symmetry and the property $q(x,y)=0 \Leftrightarrow x=y$ are
  clear. The quasi-triangle inequality
  (\ref{eq:qmetric}) follows from Lemma~\ref{lem:mprops}~\ref{item:mprops1}.

  It is well known (see \cite[Proposition~14.5]{He}) that a
  sufficient ``snowflaking'' of a quasimetric leads to a distance
  function that is  comparable to a metric. This means 
there is $0<\epsilon<1$, and a
  metric $\varrho$ on $S^2$   such that $\varrho \asymp q^\epsilon
  $.  Then $\varrho$ is a visual metric for $f$ (with expansion
  factor $\Lambda^\eps$).  
  
  \smallskip
  \ref{item:vistop}
  Let $\varrho$ be a visual metric for $f$ satisfying
  \eqref{visual}, and $d$ a fixed  base  metric  on $S^2$ that
  induces the given   topology of $S^2$. We have to show that if
  $x\in S^2$ and  $\{x_i\}$ is a sequence in $S^2$, then
  $\varrho(x_i,x)\to 0$ if and only if $d(x_i, x)\to 0$ as $i\to
  \infty$.

  Assume first that $\varrho(x_i,x)\to 0$ as $i\to \infty$. 
  By \eqref{visual} this is obviously 
  equivalent to $m_i\coloneqq m_{f,\CC}(x_i,x)\to \infty$. 
 For each $i$ there are non-disjoint $m_i$-tiles $X_i$ and $Y_i$ with $x\in X,
  x_i \in Y$. Thus
  $$ d(x_i, x) \leq \diam_d (X_i) + \diam_d (Y_i) \le 2
  \mesh(f,m_i,\CC).$$ 
  Since $f$ is expanding and $m_i\to \infty$,  the latter expression approaches  $0$ as
  $i\to \infty$. Hence $d(x_i, x)\to 0$ as $i\to \infty$.  
  
   Conversely, suppose that $d(x_i, x)\to 0$ as $ i\to \infty$.
  Let $n\in \N_0$ be arbitrary. Then $x$ lies in some $n$-flower 
  $W^n(p)$ (see Lemma~\ref{lem:mapflowers}~\ref{item:mapflowers4}). Since flowers are open sets, we have
  $x_i\in W^n(p)$ for 
  sufficiently large $i$. For each of these $i$ we can find $n$-tiles
  $X$ and $Y$ with $x\in X$, $x_i\in Y$, and $p\in X\cap Y$. This
  implies $m_i\ge n$. Therefore $m_i\to \infty$, and hence $\varrho(x_i,
  x)\to \infty$ as desired. 
 
  \smallskip
  \ref{item:visCC}
  This  follows from Lemma~\ref{lem:mprops}~\ref{item:mprops3}.
 
  \smallskip
  \ref{item:visd1d2}
  This  follows from \ref{item:visCC} and the definition of a visual
  metric.  
 
  \smallskip
  \ref{item:vsfF}
  This follows from \ref{item:visCC} and
  Lemma~\ref{lem:mprops}~\ref{item:mprops4}.  
  
   \smallskip
  \ref{item:vsLip} This follows from Lemma~\ref{lem:mprops}~\ref{item:mprops2}.
 \end{proof}

If two  expanding  Thurston maps are topologically conjugate, then
their visual metrics are closely related. 

\begin{prop} 
  \label{prop:conjisom} 
  \index{snowflake equivalent}
  Let $f\:S^2\ra S^2$ and $g\:\widehat S^2\ra \widehat S^2$ be
  expanding Thurs\-ton maps that are topologically conjugate.  Then
  $S^2$ equipped with any visual metric for  $f$ is
  snowflake equivalent to $\widehat S^2$ equipped with any
  visual metric for  $g$.  Every homeomorphism
  $h\: S^2\ra \widehat S^2$ satisfying $h\circ f=g\circ h$ is a
  snowflake equivalence.
\end{prop}

\begin{proof} By our assumptions    there exists a
  topological conjugacy between $f$ and $g$, i.e., a
  homeomorphism $h\: S^2\ra \widehat S^2$ such that
  $h\circ f=g\circ h$.  Let $\varrho$ be a visual metric on
  $S^2$ for  $f$, and $\widehat{\varrho}$ be a visual
  metric on $\widehat S^2$ for  $g$.  Let $\Lambda>1$
  and $\widehat \Lambda>1$ be the expansion factors of $\varrho$
  and $\widehat{\varrho}$, respectively.  It suffices to show
  that $h\: (S^2, \varrho)\ra (\widehat S^2, \widehat{\varrho})$
  is a snowflake equivalence.

  To see this, pick a Jordan curve $\CC\sub S^2$ with
  $\post (f)\sub \CC$. Then $\widehat \CC=h(\CC)$ is a Jordan
  curve in $\widehat S^2$ with
  $\post(g)=h(\post(f))\sub \widehat \CC$ (see Lemma~\ref{lem:T-eq_crit_post}).  Since $h$ conjugates
  $f$ and $g$, it follows from
  Proposition~\ref{prop:celldecomp}~\ref{item:skeletons} and
  \ref{item:nedgesC} or, alternatively, from the uniqueness
  statement in Lemma~\ref{lem:pullback} that for each
  $n\in \N_0$ the images of the cells in the cell decomposition
  $\DD^n\coloneqq \DD^n(f,\CC)$ of $S^2$ under the map $h$ are precisely
  the cells in the cell decomposition
  $\widehat \DD^n\coloneqq \DD^n(g, \widehat \CC)$ of $\widehat S^2$; so
  we have
\begin{equation}  \label{eq:CChCC}
\widehat \DD^n=\{h(c): c\in \DD^n\}
\end{equation} for all $n\in \N_0$. 
This implies that  
$$\widehat m(h(x),h(y))=m(x,y)$$ for all $x,y\in S^2$, where 
$\widehat m=m_{g, \widehat \CC}$ and $m=m_{f, \CC}$ (recall
Definition~\ref{def:mxy}). Combining this with
Proposition~\ref{prop:visualsummary}~\ref{item:visCC} we see
that
\begin{equation*}
  \widehat{\varrho}(h(x), h(y))
  \asymp 
  \widehat \Lambda^{-\widehat m(h(x),h(y))}
  = 
  \widehat\Lambda^{-m(x, y)}
  = 
  \Lambda^{-\alpha m(x, y)}\asymp \varrho(x,y)^\alpha 
\end{equation*}
for all $x,y\in S^2$, where
$\alpha=\log(\widehat \Lambda)/\log(\Lambda)$ and the implicit
multiplicative constants do not depend on $x$ and $y$. It
follows that $h$ is a snowflake equivalence.
\end{proof}

We now prove the geometric characterization of visual metrics.

\begin{proof}[Proof of Proposition~\ref{lem:expoexp}]
  Let $f\colon S^2 \to S^2$ be 
an expanding Thurston map, and
  $\CC\subset S^2$ be a Jordan curve with $\post(f)\subset \CC$.

We first show that a visual metric $\varrho$ for $f$ has the properties \ref{item:expoex1} and 
\ref{item:expoex2}. 
  By Proposition~\ref{prop:visualsummary}~\ref{item:visCC} we may assume that
  $\varrho$ satisfies \eqref{visual} for $m=m_{f,\CC}$ and a constant
  $C=C(\asymp)$.   

  \smallskip
  \ref{item:expoex1}
  Let $k_0\in \N$ be defined as in \eqref{def:k0}, and  let $\sigma$
  and $\tau$ be disjoint $n$-cells. If $x\in \sigma$ and $y\in \tau$
  are arbitrary, then $m(x,y)<n+k_0$. Indeed, if this were   not the
  case, then we could find $(n+k)$-tiles $X$ and $Y$ with $x\in X$,
  $y\in Y$, $X\cap Y\ne \emptyset$, and $k\ge 
k_0$.  Then $K=X\cup Y$ is a connected set  
meeting disjoint $n$-cells. Hence by Lemma~\ref{lem:flowerbds} the
number of $(n+k)$-tiles in 
$K$ should be   $\ge D_k\ge \widetilde D_{k_0}\ge 10$.  
  On the other hand, $K$ consists of  two $(n+k)$-tiles. 
   This is impossible.  

Thus
$\varrho(x,y) \ge (1/C)\Lambda^{-n-k_0}$, and so 
we get the desired bound $\dist_\varrho(\sigma, \tau)\ge (1/C')\Lambda^{-n}$ with  the constant $C'=C\Lambda^{k_0}$ that is independent of $n$, $\sigma$,  and $\tau$. 

\smallskip
\ref{item:expoex2}
If $x,y$ are points in some $n$-tile $X$, then $m(x,y)\ge n$. Since every $n$-edge is contained in an $n$-tile, this inequality is still true if $x$ and $y$ are contained in an $n$-edge.
Hence $\varrho(x,y)\le C\Lambda^{-m(x,y)}\le  C \Lambda^{-n}$, and so  
$\diam_\varrho(\tau)\le C \Lambda^{-n}$ whenever $\tau$ is an $n$-tile or $n$-edge, where $C=C(\asymp)$ is the constant from \eqref{visual}.

A similar lower bound for the diameter of an $n$-edge or $n$-tile $\tau$ follows from \ref{item:expoex1} and the fact that every $n$-edge or $n$-tile contains two distinct $n$-vertices. 

To prove the converse implication, suppose  that $\varrho$ is a metric
on $S^2$ with the  properties \ref{item:expoex1} and
\ref{item:expoex2} as in the statement. We want to show that $\varrho$ is visual for $f$. Let $x,y\in
S^2$, $x\ne y$, be arbitrary, and $m=m_{f,\CC}(x,y)$.  

Then we can find $m$-tiles $X$ and $Y$ with $x\in X$, 
$y\in Y$, and
$X\cap Y\ne \emptyset$. By \ref{item:expoex2} we have 
$$ \varrho(x,y)\le \diam_\varrho(X)+\diam_\varrho(Y) \lesssim \Lambda^{-m}. $$ 
We can also find $(m+1)$-tiles $X'$ and $Y'$ with $x\in X'$, $y\in
Y'$. By definition of $m$ we then have $X'\cap
Y'=\emptyset$. Hence by \ref{item:expoex1} 
$$\varrho(x,y)\ge \dist_\varrho(X',Y')\gtrsim \Lambda^{-m}. $$
Since the implicit multiplicative constants in the previous
inequalities are independent of $x$ and $y$, it follows that $\varrho$
is a visual metric for $f$.  
\end{proof}

It is possible to establish the phenomenon of  ``exponential
shrinking'' as in Proposition~\ref{lem:expoexp}~\ref{item:expoex2}  also for
other types of sets. For example, we have  
\begin{equation}\label{eq:flowerexpsht}
\diam_\varrho(W^n(p)) \le C \Lambda^{-n}
\end{equation}
for every
$n$-flower for $(f,\CC)$ where the constant $C$ is independent of $n$ and $p$.  Of particular importance will be exponential shrinking for lifts of paths. 

\begin{lemma}
  \label{lem:liftpathshrinks} Let $f\:S^2\ra S^2$ be an expanding Thurston map, and  
  $\varrho$ be  a visual metric for $f$ with expansion factor $\Lambda>1$. Then for every path
    $\gamma\colon [0,1]\to S^2$ there exists a
    constant $A>0$
    with the following property: if $n\in \N$ and  $\widetilde \gamma$  is any lift  of
    $\gamma$ by $f^n$,
    then
      \begin{equation*}
    \diam_\varrho (\widetilde \gamma) \le  A \Lambda^{-n}.
  \end{equation*}
\end{lemma}

Note that lifts of $\ga $ by  $f^n$ exist according to  Lemma~\ref{lem:liftsofpathsbranched}.

\begin{proof} 
  Pick a Jordan curve $\CC\sub S^2$ with $\post(f)\sub \CC$, and let
  $\delta_0>0$ be as in \eqref{defdelta} with $\varrho$ as the base metric on $S^2$.  Then 
  we can break up
  $\gamma$ into a finite number of paths $\gamma_i$, $i=1, \dots, N$,
  traversed in successive order such that $\diam_\varrho(\gamma_i)<\delta_0$
  for  $i=1,\dots, N$. By Lemma~\ref{lem:preimsmall}~\ref{item:preimsmall2} each lift
  of a piece $\ga_i$ is contained in one $n$-flower, and so the
  whole lift $\widetilde \ga$ in $N$ $n$-flowers. Hence by  
  \eqref{eq:flowerexpsht} we have $\diam_\varrho
  (\widetilde\gamma) \leq C N 
  \Lambda^{-n}$, where $C>0$ is independent of $n$ and $\ga$. The statement follows with 
  $A=CN$. 
\end{proof}

In  general, the constant $A$ in the last lemma will depend on $\ga$, but the proof shows that we can take the same constant $A$  for a family of paths if  there exists $N\in \N$ such that each 
path can be broken up into at most $N$ subpaths of diameter 
$<\delta_0$.

Let $f$ be an expanding Thurston map, and $\CC\sub S^2$ be a Jordan curve
with $\post(f)\sub \CC$. 
It is useful    to define neighborhoods of points by using the cells in our decompositions
 $\DD^n=\DD^n(f, \CC)$. To do this, let $x\in S^2$, $n\in \N_0$, and set  
\begin{align}
  \label{eq:defUk}
  U^n(x) 
  = \bigcup\{Y\in \X^n: &\text{ there exists an $n$-tile $X$ with}  
  \\ \notag 
  \null &\text{ $x\in X$ and }X\cap Y\neq\emptyset\}.  
\end{align}
It is convenient to define $U^n(x)$ also for negative integers
$n$. We set $U^n(x)=U^0(x)=S^2$ for $n<0$. 

The sets $U^n(x)$
resemble  
metric balls defined in terms of a visual metric very closely. 
\begin{lemma}
  \index{visual metric}
  \index{metric!visual}
  \label{lem:UmB} Let $\varrho$ be a visual metric for $f$ with
  expansion factor $\Lambda>1$. Then there are constants $K\ge 1$ and
  $n_0\in \N_0$  with the following properties. 
 
    \begin{enumerate}
  \item
    \label{item:UmB1}
    For all $x\in S^2$ and all $n\in \Z$,
    \begin{equation*}
      B_\varrho(x, r/K)\subset U^n(x)\subset B_\varrho(x, Kr), 
    \end{equation*}
    where  $r=\Lambda^{-n}$.
  \item
    \label{item:UmB2}
    For all $x\in S^2$ and all $r>0$,
  \begin{equation*}
     U^{n+n_0}(x)\subset B_\varrho(x,r)\subset U^{n-n_0}(x),
  \end{equation*}
  where $n=\left\lceil-\log r/\log\Lambda\right\rceil$.
  \end{enumerate}
\end{lemma}

\begin{proof} 
  \ref{item:UmB1}
  Let $m=m_{f,\CC}$. If $y\in U^n(x)$, then $m(x,y)\ge
  n$, and so $\varrho(x,y)\lesssim \Lambda^{-n}=r$.  This gives the
  inclusion $U^n(x)\sub B_\varrho(x,Kr)$ for a suitable constant $K$
  independent of $x$ and $n$.  

Conversely, suppose that  $y\notin U^n(x)$. Then $n\ge 1$. If we 
pick any $n$-tiles $X$ and $Y$ with $x\in X$ and $y\in Y$, then $X\cap
Y=\emptyset $ by definition of $U^n(x)$. So  by
Proposition~\ref{lem:expoexp}~\ref{item:expoex1} we have  
\begin{equation*}
  \varrho(x,y)\ge \dist_\varrho(X,Y) \gtrsim \Lambda^{-n}= r. 
\end{equation*}
 Hence $B_\varrho(x,r/K)\sub U^n(x)$ if $K$ is 
 suitably large independent of $x$ and $n$. 

\smallskip
\ref{item:UmB2}
Choose $n_0=\left\lceil\log K/\log \Lambda\right\rceil+1$, where $K$ is as in \ref{item:UmB1}. 
Then $\Lambda^{-n_0}\le1/( \Lambda K)$. Moreover, 
 $\Lambda^{-n}\le r\le \Lambda \Lambda^{-n}$,  and  so 
$$K\Lambda^{-n-n_0}\le r\le 
(1/K)\Lambda^{-n+n_0}. $$
The desired inclusion then follows from \ref{item:UmB1}. 
  \end{proof}

We next show that when $S^2$ is equipped with a visual metric, 
then tiles are 
``quasi-round''. 
In particular,
every tile contains points that
are ``deep inside'' the tile.

\begin{lemma} \label{lem:quasiball}
Let $f\: S^2\ra S^2$ be an expanding Thurston map, $\CC\sub S^2$ be a
Jordan curve with $\post(f)\sub \CC$, and $\varrho$ be a visual metric
for $f$ with expansion factor $\Lambda>1$.  Then there exists a
constant $C\ge 1$ with the following property:  
for every $n$-tile $X$ for $(f,\CC)$ there exists a point $p\in X$ such that 
$$B_\varrho(p, (1/C)\Lambda^{-n})\sub X \sub B_\varrho(p, C\Lambda^{-n}).$$ 
\end{lemma}

\begin{proof} With a suitable constant $C$ independent of $n$, an inclusion of the form 
$$X\sub B_\varrho(p, C\Lambda^{-n})$$  holds for every $n$-tile  $X$ and
every point $p\in X$ as follows from Proposition~\ref{lem:expoexp}~\ref{item:expoex2}. 

The main difficulty for an inclusion in the opposite direction is to find an  appropriate  point $p$. For this purpose, let $k_0\in \N$ be the number defined in \eqref{def:k0}, and $X$ be an arbitrary $n$-tile. 
Since $f$ is an expanding  Thurston map, we have $\#\post(f)\ge 3$ (see Lemma~\ref{lem:no<3}), and so 
$\partial X$ contains at least three distinct $n$-vertices 
$v_1,v_2,v_3$. Using these vertices, we can find three arcs
$\alpha_1, \alpha_2,\alpha_3\sub \partial X$ with pairwise disjoint interior such that    
$\partial X= \alpha_1\cup \alpha_2\cup \alpha_3$ and such 
that $\alpha_i$ has the endpoints $v_i$ and $v_{i+1}$ for $i=1,2,3$, where $v_4=v_1$.    In general, $\alpha_i$ will not be an $n$-edge, but since  it lies on $\partial X$, and its endpoints are $n$-vertices, it is the union of all the $n$-edges that it contains. 

We now define 
$$A_i=\bigcup_{x\in \alpha_i} U^{n+k_0}(x) $$
for $i=1,2,3$, where $U^{n+k_0}(x)$ is given  as in \eqref{eq:defUk}.
Then  the set $A_i$ is the union of all $(n+k_0)$-tiles that meet an $(n+k_0)$-tile that has non-empty  intersection with $\alpha_i$. In particular, $A_i$ is a closed set that contains $\alpha_i$. 

We claim that the sets $A_1$, $A_2$, $A_3$ do not form a cover of $X$. To reach a contradiction, suppose that 
$X\sub A_1\cup A_2\cup A_3$. We can regard $X$ as a topological simplex with the sides $\alpha_i$, $i=1,2,3$. 
Then the  closed sets $A_1, A_2, A_3$ form a cover of $X$ such that each set $A_i$ contains the side $\alpha_i$ of the simplex for $i=1,2,3$. A well-known result due to Sperner \cite[p.~378]{AH} then implies that $A_1\cap A_2\cap A_3\ne \emptyset$.

Pick a point $x\in A_1\cap A_2\cap A_3$.  Then by definition of $A_i$, there exist $(n+k_0)$-tiles $X_i$ and $Y_i$ with 
$X_i\cap \alpha_i\ne \emptyset$, $x\in Y_i$, and $X_i\cap Y_i\ne \emptyset$, where $i=1,2,3$. Then the set 
$$K=\bigcup_{i=1}^3(X_i\cup Y_i)$$
consists of at most six  $(n+k_0)$-tiles, is connected,  and meets each of the arcs $\alpha_1,\alpha_2, \alpha_3$.
Hence $K'=f^n(K)$ is a connected set that consists of at most 
six  $k_0$-tiles, and meets each of the arcs $\beta_i=f^n(\alpha_i)$,
$i=1,2,3$. Note that  each arc  $\beta_i$ is the union of all $0$-edges that it contains. Hence for $i=1,2,3$ there exists a  $0$-edge $e_i\sub \beta_i$ with $e_i\cap K' \ne \emptyset$. Since the arcs $\beta_1,\beta_2,\beta_3$  have pairwise disjoint interior, it follows that the $0$-edges $e_1,e_2,e_3$ are all distinct. So $K'$ is a connected set that meets three distinct $0$-edges. Hence it joins opposite sides of $\CC$. So $K'$ should contain at least $D_{k_0}\ge \widetilde D_{k_0}\ge 10$ tiles of level $k_0$.
This is a contradiction, because $K'$ is a union of at most six  $k_0$-tiles. 

This proves the claim that the sets $A_1,A_2, A_3$ do not cover $X$, and we conclude that we can find a point 
$$p\in X\setminus (A_1\cup A_2\cup A_3).$$

We claim that $U^{n+k_0}(p)\sub X$. Otherwise,  
there is a point    
$y\in U^{n+k_0}(p)\setminus X$, and  $(n+k_0)$-tiles 
$U$ and $V$ with $p\in U$, $y\in V$, and $U\cap V\ne \emptyset$. 
Then the connected set $U\cup V$ must meet $\partial X$, and hence one of the arcs $\alpha_i$; but then $p\in A_i$
by definition of $A_i$. This is a contradiction showing 
the desired inclusion $U^{n+k_0}(p)\sub X$.

From  Lemma~\ref{lem:UmB}~\ref{item:UmB1} it now follows   that 
 $B_\varrho(p,(1/C)\Lambda^{-n})\sub X$, where $C\ge 1$ is a constant
 independent of $n$ and $X$.  
 \end{proof}

 \section{The canonical orbifold metric as a visual metric}
\label{sec:orbifold_visual}


If $f\: S^2\ra S^2$ is an expanding Thurston map $f$, one can ask whether 
other standard metrics on $S^2$ are visual metrics. To have some natural 
metrics available, we restrict ourselves here to rational expanding
Thurston maps $f$ defined on $\CDach$. Recall that a rational
Thurston map is expanding if and only if it 
does not have periodic critical points 
(see Proposition~\ref{prop:rationalexpch}). 

\begin{lemma}
  \label{lem:chord_not_vis}
  \index{metric!chordal}
  \index{chordal metric}
  Let $f\colon \CDach\to \CDach$ be a rational expanding Thurston
  map. Then the 
  chordal metric $\sigma$ on $\CDach$ is not a visual metric for $f$. 
\end{lemma}

\begin{proof}
  We argue by contradiction and assume that $\sigma$ is a visual metric for  $f$ with expansion factor $\Lambda>1$. We pick a  critical point $c\in
  \CDach$ of  $f$  and set 
  $d\coloneqq \deg(f,c)\geq 2$. 
  
  Consider tiles for $(f,\CC)$, 
  where  
 $\CC\subset \CDach$ is a fixed Jordan curve with
 $\post(f)\subset \CC$ as usual. 
 For
  each $n\in \N$ let $X^n$ be an $n$-tile that contains $c$. Then   $\diam_\sigma(X^n) \asymp \Lambda^{-n}$  by
  Proposition~\ref{lem:expoexp}~\ref{item:expoex2}.  Since in suitable 
  local   conformal coordinates the map $f$ near $c$ behaves like $z\mapsto z^d$ near $0$, for $Y^n\coloneqq f(X^{n+1})$ we have
   \begin{align}\label{eq:uppbddd}
   \diam_\sigma(Y^n) &=  \diam_\sigma(f(X^{n+1}))  \asymp
   \diam_\sigma(X^{n+1})^d  \\
   & \asymp \Lambda^{-(n+1)d} \asymp \Lambda^{-nd}\nonumber
  \end{align}
  for $n\in \N_0$, where $C(\asymp)$ is independent of
  $n$. On the other hand, $Y^n$  is an $n$-tile 
 and so $\diam_\sigma(Y^n)\asymp \Lambda^{-n}$ by
 Proposition~\ref{lem:expoexp}~\ref{item:expoex2}, where
 $C(\asymp)$ is independent of  $n$. Since $d\ge 2$ this is irreconcilable with \eqref{eq:uppbddd} for large $n$ and so we reach a contradiction. 
\end{proof}

We now turn to the canonical
  orbifold metric $\omega=\omega_f$  of  a given rational expanding Thurston map $f$ (see Section~\ref{sec:orbif-assoc-thurst})   and its relation to visual metrics. We will reformulate and prove   the  two implications in Proposition~\ref{prop:orbivispara} separately. 
  First we prove the ``only if'' statement.

\begin{lemma}
  \label{lem:fhyp_om_not_visual}
  Let $f\colon \CDach \to \CDach$ be a rational expanding Thurston map
  such that its canonical orbifold metric
  $\omega=\om_f$ is a visual metric for $f$. Then $f$ is a Latt\`{e}s map. 
\end{lemma}

\begin{proof} The metric $\om$ is the canonical orbifold metric
  of the associated orbifold $\mathcal{O}_f=(\CDach,
  \alpha_f)$.
  Note that $\alpha_f(u)<\infty$ for $u\in \CDach$ by
  Proposition~\ref{prop:otherramprops}~\ref{item:rami_infty},
  because $f$ is a rational expanding Thurston map and so it does
  not have periodic critical points.

  Suppose $\om$  is a visual metric for $f$ with 
  expansion factor $\Lambda>1$. 
  Let $p\in \CDach$ 
and $q\in f^{-1}(p)$ be arbitrary, and set $d\coloneqq \deg_f(q)$.  
 
 We consider tiles for $(f,\CC)$, where  
 $\CC\subset \CDach$ is a fixed Jordan curve with $\post(f)\subset \CC$. 
 For each $n\in \N_0$ we pick an  $n$-tile $X^n$  with $q\in X^n$. Then $Y^n\coloneqq f(X^{n+1})$ is an $n$-tile containing $p=f(q)$. 
  
  As in the proof of Lemma~\ref{lem:chord_not_vis}, we have 
  $$ \diam_\sigma(Y^n)\asymp \diam_\sigma(X^{n+1})^d. $$ 
  On the other hand, one can relate $\om$ and the chordal metric $\sigma$. Namely, one can show (see \eqref{eq:om_chordal})
  that if $U$ is a sufficiently small neighborhood of $p$, then 
  $$ \om( p, u) \asymp\sigma(p, u)^{1/\alpha_f(p)}$$ 
  for all $u\in U$.
  This implies that 
  $$ \diam_\om (Y^n) \asymp  \diam_\sigma(Y^n)^{1/\alpha_f(p)} $$ 
  for large $n$. 
  Similarly, 
  $$  \diam_\om (X^n) \asymp  \diam_\sigma(X^n)^{1/\alpha_f(q)}$$
  for large $n$. 
  
  Since $\omega$ is a visual metric for $f$ with expansion factor $\Lambda$, we also have 
  $$ \diam_\omega(X^n) \asymp \diam_\omega(Y^n) \asymp \Lambda^{-n}$$ for 
  $n\in \N_0$ by Proposition~\ref{lem:expoexp}~\ref{item:expoex2}. 
  
  If we combine all these  estimates, we arrive at 
  \begin{align*}  
    \Lambda^{-\alpha_f(p)n}
    & \asymp  
      \diam_\omega(Y^n)^{\alpha_f(p)}
   \asymp  
      \diam_\sigma(Y^n)
    \\ 
    & \asymp 
      \diam_\sigma(X^{n+1})^d
      \asymp   
      \diam_\omega(X^{n+1})^{d\alpha_f(q)}
    \\ 
    & \asymp
      \Lambda^{-d\alpha_f(q)(n+1)}
      \asymp
      \Lambda^{-d\alpha_f(q)n}
  \end{align*} 
  for all large $n$, where all the implicit constants $C(\asymp)$ are independent of $n$.  This is only possible if $\alpha_f(p)=d  \alpha_f(q)=\deg_f(q) \alpha_f(q)$.
  
We conclude that 
  $\alpha_f$ satisfies  condition~\ref{item:Of_para3} in  Proposition~\ref{prop:parabolicOf}, and so 
  $\mathcal{O}_f$ is parabolic. Hence $f$ is a Latt\`es map by 
  Theorem~\ref{thm:Lattesstruc}. 
  \end{proof}

\index{canonical orbifold!metric}\index{metric!canonical orbifold}\index{o@$\omega$}\index{orbifold!canonical metric} 
We now prove the ``if'' implication of
Proposition~\ref{prop:orbivispara}.
\begin{prop} 
  \label{prop:paravisorbmet}
  \index{Latt\`{e}s map} 
  Let $f\: \CDach\ra \CDach$ be a Latt\`es map,
  and $\om=\om_f$ be the canonical orbifold metric of $f$ on $\CDach$.  Then
  $\om$ is a visual metric for $f$ with expansion factor $\Lambda=\deg
  (f)^{1/2}>1$.  Moreover, $\om$ is a geodesic metric on $\CDach$ and
  for all paths $\ga$ in $\CDach$ we have
\begin{equation}
  \label{eq:expparaorbmet} 
  \length_\om(f\circ \ga)=\Lambda\length_\om  (\ga).
\end{equation}
\end{prop}

We will see later (in Proposition~\ref{prop:macgrdn} and
Theorem~\ref{thm:visexpfactors1}) that for given degree of an
expanding Thurston map $f$ the number $\Lambda=\deg (f)^{1/2}$ is the
largest possible expansion factor of a visual metric. So Latt\`es maps
are special as they realize this maximal factor. This is closely
related to the characterization of Latt\`es maps among expanding
Thurston maps  given in Theorem~\ref{thm:Qianthm}.

\begin{proof} 
Let  $f\colon \CDach \to \CDach$ be a Latt\`{e}s map. Then $\om$ is the canonical 
orbifold metric of the parabolic orbifold $\mathcal{O}=(\Cdach, \alpha_f)$ associated with $f$.
Since $f$ has no periodic critical points, we have $\alpha_f(p)<\infty$ for all $p\in \CDach$ (Proposition~\ref{prop:otherramprops}~\ref{item:rami_infty}) and so  
$\mathcal{O}$ has no punctures. In particular, $\om$ is a
geodesic metric defined on the whole Riemann sphere $\CDach$ (see the discussion
after \eqref{eq:om_path_isom} in Section~\ref{sec:expratThmaps}).

Let $\Theta\: \C\ra \CDach$  and $A\: \C \ra \C$ be holomorphic maps for $f$ as 
in Theorem~\ref{thm:Lattesstruc}~\ref{item:Lattessruciii}.  Then 
$f\circ \Theta = \Theta \circ A$, and $A$ has the form  $A(z)= \alpha z+\beta$, where
$\alpha,\beta\in \C$ with $\deg(f)=\abs{\alpha}^2\ge 2$ (see Lemma~\ref{lem:deglatttype}). 
We  know that $\Theta$ is the universal orbifold covering map of $\mathcal{O}$ and
that the metric $\om$ is essentially the push-forward of the Euclidean metric on $\C$ 
by the map $\Theta$. More precisely, $\Theta$ is a path isometry in the sense that   
\begin{equation}\label{eq:piso8} \length(\ga)=
\length_\omega(\Theta\circ \ga),
\end{equation} whenever $\ga$ is a path in $\C$
(see \eqref{eq:om_path_isom}). Here and below metric notions on $\C$ (such as 
$\length(\gamma)$) refer to the Euclidean metric, while on $\CDach$ we use the metric $\om$ indicated by a subscript.  

Define $\Lambda = \abs{\alpha} = \deg(f)^{1/2}>1$ and let $\ga$ be an arbitrary path in $\CDach$. Then $\ga$ has a lift
by the branched covering map $\Theta$, and so there exists a path
$\widetilde {\gamma}$ in $\C$ such that $\ga =\Theta\circ
\widetilde {\gamma}$ (see Lemma~\ref{lem:liftsofpathsbranched}). Then
$f\circ \Theta = \Theta\circ A$ implies
$$ \Theta\circ A^n\circ \widetilde {\gamma} = f^n \circ \Theta \circ
\widetilde {\gamma} = f^n\circ \ga$$
for all $n\in \N$.

Note
that $A^n$ for $n\in \N$ is a Euclidean similarity on $\C$
scaling distances by
the factor $\abs{\alpha}^n =\Lambda^n$.
Since $\Theta$ is a path isometry,  we conclude that 
\begin{align}
  \label{eq:Lattscalfn}
  \length_\om(f^n\circ \ga)&=\length(A^n\circ \widetilde {\gamma})
  \\
  &=\Lambda^n\length (\widetilde {\gamma})=\Lambda^n\length_\om  (\ga). 
  \nonumber 
\end{align} 
If we take $n=1$ here, then  \eqref{eq:expparaorbmet} follows.

In order  to verify  that  $\omega$
is a visual metric for $f$ with expansion factor $\Lambda$, 
we will show that $\om$ satisfies the conditions in 
Proposition~\ref{lem:expoexp}.
For this, we  pick a Jordan curve $\CC\sub \CDach$ with $\post(f)\sub \CC$ and consider 
  $n$-cells for $(f,\CC)$.   
  
  Suppose $\sigma$ and $\tau$ are two disjoint $n$-cells, where $n\in \N_0$.
Since  $\omega$ is a geodesic metric, we can find a path $\ga$ in
$\CDach$ joining  
$\sigma$ and $\tau$
with 
\begin{equation}\label{eq:visLattscal0}
 \length_\om(\ga)= \dist_\om(\sigma, \tau). \end{equation} 
By Lemma~\ref{lem:maptotop} the path $f^n\circ \ga$ joins opposite sides of $\CC$, and so 
$$ \length_\om(f^n\circ \ga)\ge \diam_\om(f^n\circ \ga)\ge \delta_0, $$
where $\delta_0>0$ is defined as in \eqref{defdelta} for the base metric $\om$
on $\CDach$.

Combining  the last inequality with \eqref{eq:visLattscal0} and 
 \eqref{eq:Lattscalfn}, we arrive at the inequality 
 \begin{equation} 
   \label{eq:visLattscal1} 
   \dist_\om(\sigma, \tau)\gtrsim \Lambda^{-n}. 
 \end{equation}
 Here and in the following, the implicit multiplicative constant
 $C(\gtrsim)$ is independent of the cells and their level $n$. 
So $\om$ satisfies condition \ref{item:expoex1} in
Proposition~\ref{lem:expoexp}.

Since every $n$-edge or $n$-tile $\tau$   contains two distinct $n$-vertices, \eqref{eq:visLattscal1}
also implies that 
\begin{equation}\label{eq:visLattscal3}
 \diam_\om(\tau)\gtrsim \Lambda^{-n}. 
 \end{equation}  

To complete the proof, we have 
 to establish   an inequality in the opposite  direction.
  First note that by  Lemma~\ref{lem:unifcontTheta} we can 
find a  number $\delta_1>0$ with the following property: 
 if  $K\sub \C$ is a connected set with 
$\diam_{\om}(\Theta(K))<\delta_1$, then we have $\diam(K)< 1$.
 Since $f$ is expanding, we can choose $n_0\in \N$ such that 
$\diam_\om(X)<\delta_1$, whenever $X$ is an $n$-tile with level $n\ge n_0$.

Now suppose   $X$ is   an arbitrary $n$-tile with $n\ge n_0$. Define
$k=n-n_0\in \N_0$ and $U=\inte(X)$. Then $U\subset \CDach$ is a simply
connected region that does not contain any $n$-vertex and so 
$U\sub \CDach\setminus \post(f)$. Since $\Theta\: \C\setminus \Theta^{-1}(\post(f))\ra \CDach\setminus \post(f)$ is a covering map (see Lemma~\ref{lem:brcovcov}), the inclusion map $U\ra \CDach\setminus \post(f)$ lifts by $\Theta$ (see Lemma~\ref{lem:liftsofcov}). This  implies that there exists a region 
$V\sub \C$ such that $\Theta(V)=U$. Note that $\diam_\om( U)   
\le \diam( V)$, because distances do not increase under the map $\Theta$ as follows 
from \eqref{eq:piso8}.

Moreover, 
\begin{equation} \label{eq:visLattscal2}
\Theta(A^k(V))=f^k(\Theta(V))=f^k(U)\sub f^k(X). 
\end{equation} 
On the other hand, $f^k(X)$ is a tile of level $n-k=n_0$, and so 
$$\delta_1> \diam_\om(f^k(X))\ge \diam_\om(f^k(U)). $$
Since  $A^k(V)$ is connected, this inequality, the relation \eqref{eq:visLattscal2}, and the definition of 
$\delta_1$ imply that 
$$ \diam(A^k(V))\le 1.$$ 
It follows that 
\begin{align*} \diam_\om(X) &=\diam_\om(\overline U)=
\diam_\om( U)   
\le \diam( V)\\ &=\Lambda^{-k}\diam(A^k(V))  \le \Lambda^{-k}=
\Lambda^{n_0}\Lambda^{-n}\lesssim \Lambda^{-n}. \nonumber
\end{align*}

For $n$-tiles $X$ with $n\le n_0$ we trivially have $\diam_\om(X)\asymp 1\asymp \Lambda^{-n}$. Hence 
$\diam_\om(X)\lesssim \Lambda^{-n}$ for tiles $X$ on  any level $n\in \N_0$ if we choose a suitable constant $C(\lesssim)$.
Since every $n$-edge is contained in an $n$-tile,  we have 
$ \diam_\om(\tau)\lesssim \Lambda^{-n}$ whenever 
$\tau$ is an $n$-tile or $n$-edge, $n\in \N_0$. Here we actually have  
$ \diam_\om(\tau)\asymp \Lambda^{-n}$  by \eqref{eq:visLattscal3}. 
 
This shows that $\om$ also satisfies  condition~\ref{item:expoex2}  in Proposition~\ref{lem:expoexp}. It follows that
$\om$ is indeed  a visual metric for $f$ with expansion factor $\Lambda$. 
\end{proof}

Note that Proposition~\ref{prop:orbivispara} clearly follows from  
Proposition~\ref{prop:paravisorbmet} and
Lemma~\ref{lem:fhyp_om_not_visual}.

\ifthenelse{\boolean{singlechapter}}{

%


\chapter{Symbolic dynamics}
\label{cha:symdym}

\index{symbolic dynamics}

If one wants to understand a dynamical system $(X,f)$ given by the iteration  of a map $f$ on a space $X$, then often one tries to find a link to symbolic dynamics 
and the theory of  shift operators, in particular  to shifts of finite type.
These  operators serve as an  important paradigm in dynamics. In this chapter we study this for
expanding Thurston maps, but  we will exclusively be  concerned with topological aspects. One could also investigate measure-theoretic 
properties of expanding Thurston maps and their relation to (Bernoulli) shift operators, but we will not pursue this here (see \cite{HH} for a relevant paper in this context).   

 We will prove  the following statement.

  \begin{theorem} \label{thm:expThfactor}
Let $f\: S^2\ra S^2$ be an expanding Thurston map. Then $f$ is a factor 
of the left-shift $\Sigma\: J^\om\ra J^\om $ on the space $J^\om$ of all sequences in a finite set $J$ of cardinality $\#J=\deg(f)$. 
\end{theorem}

The notation and terminology  will be explained below.

Theorem~\ref{thm:expThfactor} is essentially  due to Kameyama (see \cite[Theorem~3.4]{Ka03a}).
The basic idea seems to go back to \cite{Jo} (see also \cite {Prz}).  
Ka\-me\-ya\-ma's notion of an expanding Thurston map is  different from ours, but his  proof carries over to 
our setting with only minor modifications.

%
%
%
%

It is a standard  fact in complex dynamics that the repelling periodic points 
of a rational map on $\CDach$ are dense in its Julia set. 
The following statement is an analog of this for expanding Thurston maps.
As we will see, it easily 
follows from the proof of Theorem~\ref{thm:expThfactor}. 

\begin{cor}\label{cor:perdense}
Let $f\:S^2\ra S^2$ be an expanding Thurston map.
Then the periodic points of $f$ are dense in $S^2$. 
\end{cor} 
\index{periodic!point}

Before we  supply the proofs of these  results, we will first   review some basic definitions related to shift operators. 

 Let $J$ be a finite non-empty set.  We consider $J$ as an {\em alphabet} and 
its elements as {\em letters} in this alphabet.  A {\em word}  is a finite sequence $w=i_1i_2\dots i_n$, where $n\in \N_0$ and $i_1, i_2,\dots, i_n\in J$. For $n=0$ we interpret this as the {\em empty word} $\emptyset$. The number $n$ is  called the {\em length} of the word 
$w=i_1i_2\dots i_n$.  The words of length $n$ can be identified 
with $n$-tuples in $J$ and are  elements in  the Cartesian power $J^n$. The letters, i.e., the elements in $J$, are precisely the words of length $1$. If $w=i_1i_2\dots i_n$ and $w'=j_1j_2\dots j_m$, then   
we denote by $ww'=i_1i_2\dots i_nj_1j_2\dots j_m$ the word obtained by concatenating $w$ and $w'$.  

Let $J^*$ be the set of all words (including the empty word) in the alphabet $J$. 
The {\em (left-)shift}\index{shift}\index{left-shift}
$\Sigma\: J^*\setminus \{\emptyset\}\ra J^*$ is defined by setting $\Sigma(i_1i_2\dots i_n)=i_2\dots i_n$ 
for a word $w=i_1i_2\dots i_n\in J^*\setminus \{\emptyset\}$. We denote by $J^\om$ 
 the set of all sequences $\{i_k\}=\{i_k\}_{k\in \N}$ 
 with $i_k\in J$ for  $k\in \N$. More informally, we consider a
 sequence $s=\{i_k\}\in J^\om$ as a word of infinite length 
and write $s=i_1i_2\dots $\,. 

If 
$s=\{i_k\}\in J^\om$ and $n\in \N_0$, then we denote by $[s]_n\in J^*$ the
word $s_n=i_1i_2\dots i_n$ consisting of the first $n$ elements of the sequence $s$.  
The {\em \mbox{(left-)}shift} $\Sigma\: J^\om \ra J^ \om$ is the map that assigns to each sequence $\{i_k\}\in J^\om $ the sequence $\{j_k\}\in J^\om$ with $j_k=i_{k+1}$ for  $k\in\N$.  In our notation we do not distinguish the shifts on $J^*\setminus\{\emptyset\}$ and $J^\om$ and denote both maps by $\Sigma$.
Note that $\Sigma([s]_{n}) =[\Sigma(s)]_{n-1}$ for $s\in J^\om$ and $n\in \N$; indeed,
if $s=i_1i_2\dots$, then we have 
$$\Sigma([s]_{n})=\Sigma(i_1i_2\dots i_{n})=i_2\dots i_{n}=[i_2i_3\dots]_{n-1} = [\Sigma(s)]_{n-1}. $$

If we equip 
$J$ with the discrete topology, then $J^\om$ carries a natural metrizable product topology. This topology is induced by the ultrametric $d$ given by
$d(s,s')=2^{-N}$ for $s=\{i_k\}\in J^\om $ and $s'=\{j_k\} \in J^\om$, $s\ne s'$, where $N=\min\{k\in \N:i_k\ne j_k\}$. In particular, two elements $s,s'\in J^\om$ are close if and only if $s_k=s'_k$ for all $k=1, \dots, n$, where  $n$ is large.  Equipped with this topology, the space $J^\om$ is compact.

Suppose that $X$ and $\widetilde X$ are topological spaces, and $f\:
X\ra X$ and $\widetilde f\: \widetilde X\ra  \widetilde X$ are
continuous maps. We say that the dynamical system $(X,f)$ is a {\em
  factor} of the dynamical system\index{factor of dynamical system}
 $(\widetilde X, \widetilde f)$ if there exists a surjective continuous map 
$\varphi\: \widetilde X\ra X$ such that $\varphi\circ \widetilde f=f\circ \varphi$.

We are now ready for the proofs.

\begin{proof}[Proof of Theorem~\ref{thm:expThfactor}] Let $f\:S^2\ra S^2$ be an expanding Thur\-ston map, and  $k\coloneqq \deg(f)\ge 2$.  Fix a visual metric $\varrho$ for $f$, and let $\Lambda>1$ be its expansion factor. In the following, metric concepts refer to $\varrho$.   

We fix a Jordan curve $\CC\sub S^2$ with $\post(f)\sub \CC$ and consider tiles for $(f,\CC)$. We  color them black and white as in Lemma~\ref{lem:colortiles}. Let    $p\in S^2\setminus \post(f)$ 
be a basepoint in the interior of the white $0$-tile $\XOw$.

\smallskip
\emph{Claim 1.}
For all $n\in \N_0$ the estimate
\begin{equation*}
  \sup_{x\in S^2} \dist(x, f^{-n}(p))
  \lesssim 
  \Lambda^{-n} 
\end{equation*}
is true,
where $C(\lesssim)$ is independent of $n$.

\smallskip  In other words, the set $f^{-n}(p)$ forms a very dense net in $S^2$ if 
$n$ is large.  To see this, let $x\in S^2$   be arbitrary. Then $x$ lies in some $n$-tile $X^n$. If $X^n$ is white, then $X^n$ contains  
 a point in $f^{-n}(p)$ and so $\dist(x, f^{-n}(p))\le \diam(X^n)$.
If $X^n$ is black, then $X^n$ shares an edge with a white  $n$-tile $Y^n$.
Then $Y^n$  contains a  point in $f^{-n}(p)$, and so  
$\dist(x, f^{-n}(p))\le \diam(X^n)+\diam(Y^n)$.

From the inequalities in  both cases and
Proposition~\ref{lem:expoexp} we 
conclude
$$ \dist(x, f^{-n}(p))\lesssim \Lambda^{-n}, $$
where $C(\lesssim)$ is independent of $x$ and $n$.
Claim~1 follows. 

\smallskip

None of the points in $S^2\setminus \post(f)$ is a critical value for any of the iterates $f^n$ of $f$. Moreover, each iterate $f^n$ is a covering 
map  $f^n\: S^2\setminus f^{-n}(\post(f))\ra S^2\setminus
\post(f)$ 
(see Lemma~\ref{lem:brcovcov}). 
Since   $p\in S^2\setminus \post(f)$, we have $f^{-n}(p)\sub S^2\setminus f^{-n}(\post(f))$  and 
\begin{equation}\label{eq:preimcard} 
\#f^{-n}(p)=\deg(f^n)=\deg(f)^n=k^n
\end{equation} 
for $n\in \N$. 
In particular, $$f^{-1}(p)\sub S^2\setminus f^{-1}(\post(f))\sub S^2\setminus \post(f),$$ and $\#f^{-1}(p)=k$. 
Let $q_1, \dots, q_k\in S^2\setminus \post(f)$  be the points in $f^{-1}(p)$. 
For each $i=1, \dots, k$ we pick a path $\alpha_i\:[0,1]\ra S^2\setminus  \post(f)$  
 with $\alpha_i(0)=p$ and $\alpha_i(1)=q_i$. 
 
 Let $J\coloneqq \{1,\dots, k\}$, and consider the shift $\Sigma\: J^\om \ra J^\om 
 $. We want to show that $f$ is a factor of $\Sigma$, i.e., that there exists a continuous and surjective map $\varphi\: J^\om \ra S^2$ with 
 $f\circ \varphi=\varphi\circ \Sigma$.  In order to define $\varphi$, we first construct a  map $\psi$ that assigns to each word in $J^*$ a point in $S^2$.

 \smallskip
 {\em Definition of $\psi$}:  The  map $\psi\: J^*\ra S^2$ will be defined inductively such that 
$$ \psi(w)\in f^{-n}(p), $$ whenever $n\in \N_0$ and $w\in J^n\sub J^*$ is a 
word of length $n$. 
For the empty word $\emptyset$ we set $\psi(\emptyset)=p$, and for the  
word consisting of the single letter $i\in J$ we set $\psi(i)\coloneqq q_i\in f^{-1}(p)$.

Now suppose that $\psi$ has been defined for all words of length $\le n$, where 
$n\in \N$. 
Let $w$ be an arbitrary word of length $n+1$. Then $w=w'i$, where 
$w'\in J^*$ is a word of length $n$ and $i\in J$. So  $\psi(w')\in f^{-n}(p)$ is already defined. Since $f^n(\psi(w'))=p$ and $f^n\: S^2\setminus 
f^{-n}(\post(f))\ra S^2\setminus \post(f)$ is a covering map, the path 
$\alpha_i$ has a unique lift with initial point $\psi(w')$, i.e., there exists 
a unique path $\widetilde \alpha_i\: [0,1]\ra S^2$
with $\widetilde \alpha_i(0)=\psi(w')$ and $f^n\circ \widetilde \alpha_i=
\alpha_i$ (see Lemma~\ref{lem:liftsofcov}). We now define 
$ \psi(w)\coloneqq \widetilde \alpha_i(1)$. 
Note that then
$$f^{n+1}( \psi(w))=f^{n+1}(\widetilde \alpha_i(1))=f( \alpha_i(1))=f(q_i)=p.$$
Hence $\psi(w)\in f^{-(n+1)}(p)$.
This shows that a map $\psi\: J^*\ra S^2$ with the desired properties
exists.  

\smallskip
{\em Claim 2.} $f(\psi(w))=\psi(\Sigma(w))$ for all non-empty words $w\in J^*$.
\smallskip 
 
We prove this by induction on the length of the word $w$. If $w=i\in J$, then
$$f(\psi(w))=f(\psi(i))=f(q_i)=p=\psi(\emptyset)=\psi(\Sigma(i))=
\psi(\Sigma(w)).$$
So the claim is true for words of length $1$.

Suppose the claim is true for words of length $\le n$, where $n\in\N$.
Let $w$ be a word of length $n+1$. Then $w=w'i$, where $w'$ is a word of 
length $n$ and $i\in J$. 
Let $\widetilde \alpha_i$ be the  path as above, used in the definition of 
$\psi(w)$. Define $\widetilde \beta_i\coloneqq f\circ \widetilde \alpha_i$. Then 
$\widetilde \beta_i$ is a lift of $\alpha_i$ by $f^{n-1}$. By induction hypothesis its initial point is 
$$\widetilde \beta_i(0)=f(\widetilde \alpha_i(0))=f(\psi(w'))=\psi(\Sigma(w')). $$
In other words, $\widetilde \beta_i$ is  the unique 
 path as in the definition of $\psi$ used to determine  $\psi(\Sigma(w')i)$ from $\psi(\Sigma(w'))$, 
and so  $\psi(\Sigma(w')i)=\widetilde \beta_i(1)$.
 Hence 
 $$\psi(\Sigma(w))=\psi(\Sigma(w')i)=\widetilde\beta_i(1)=
 f(\widetilde\alpha_i(1))=f(\psi(w'i))=f(\psi(w))$$
 as desired, and Claim 2 follows. 

\smallskip
{\em Claim 3.} For each $n\in \N$ the map $\psi|J^n\: J^n\ra f^{-n}(p)$
is a bijection.

\smallskip 
In other words, the map $\psi$ provides a coding of the points in $f^{-n}(p)$ by words of length $n$. 
Again we prove this by induction on $n$. By definition of $\psi$ it is true for 
$n=1$. 

Suppose it is true for some $n\in \N$. 
Then it is enough  
to show 
that the map $\psi|J^{n+1}\: J^{n+1} \ra f^{-(n+1)}(p)$ is surjective, since both sets $J^{n+1}$ and $f^{-(n+1)}(p)$ have the same cardinality
$k^{n+1}$. So let  $x\in f^{-(n+1)}(p)$ be arbitrary. Then 
$f^n(x)\in f^{-1}(p)$, and so there exists $i\in J$ with $f^n(x)=q_i$.
Since $$x\in f^{-(n+1)}(p)\sub S^2\setminus f^{-(n+1)}(\post(f))\sub
S^2\setminus f^{-n}(\post(f)), $$ and $f^n 
\: S^2\setminus
f^{-n}(\post(f))\ra S^2\setminus \post(f)$  is a covering map, we can
lift the path $\alpha_i$ by $f^n$ to a path $\widetilde
\alpha_i\:[0,1]\ra S^2$ whose terminal point is $x$ (to see this, 
lift   $\alpha_i$, traversed  in opposite direction, so that the initial point of the lift is $x$). 
Then $f^n(\widetilde \alpha_i(0))=\alpha_i(0)=p$, and so 
$\widetilde \alpha_i(0)\in f^{-n}(p)$. By induction hypothesis there exists a 
word $w'\in J^n$ with $\psi(w')=\widetilde \alpha_i(0)$. 
Then $\widetilde \alpha_i$ is a path as used to determine $\psi(w'i)$ from 
$\psi(w')$. So if we set $w\coloneqq w'i\in J^{n+1}$, then 
$$ \psi(w)=\psi(w'i)=\widetilde \alpha_i(1)=x.$$
This shows that $\psi|J^{n+1}\: J^{n+1} \ra f^{-(n+1)}(p)$ is surjective.   Claim~3 follows.


\smallskip
\emph{Claim 4.} If $s\in J^\om $, then the points $\psi([s]_n)$,
$n\in \N$, form a Cauchy sequence in $S^2$ (recall that $[s]_n$
is  the word consisting of the first $n$ elements of the sequence
$s$).

\smallskip
Indeed, by definition of $\psi$ the points $\psi([s]_n)$ and $\psi([s]_{n+1})$
are joined by a lift of one of the paths $\alpha_1, \dots, \alpha_k$ by 
$f^n$. So  by Lemma~\ref{lem:liftpathshrinks} we have 
\begin{equation} \label{eq:Cauest}
\varrho(\psi([s]_n), \psi([s]_{n+1}))\lesssim \Lambda^{-n}, 
\end{equation}
where $C(\lesssim)$ is independent of $n$ and $s$. 
Hence  $\{\psi([s]_n)\}$ is a Cauchy sequence, proving Claim~4.

\smallskip 
{\em Definition of $\varphi$}:  If  $s\in J^\om$, then by Claim~4 the limit
 $$\varphi(s)\coloneqq  \lim_{n\to \infty} \psi([s]_n) $$
 exists. This defines a map $\varphi\: J^\om \ra S^2$. 
 
\smallskip
\emph{Claim 5.} $f\circ \varphi = \varphi\circ \Sigma$.
\smallskip

 To see this, let $s\in J^\om $ be arbitrary. Note that 
 $\Sigma([s]_n)=[\Sigma(s)]_{n-1}$ 
 for $n\in \N$. Hence by Claim 2 and the continuity of $f$  we have 
 \begin{align*}
 f(\varphi(s))&=\lim_{n\to \infty} f(\psi([s]_n))=\lim_{n\to \infty}
 \psi(\Sigma([s]_n))\\&=\lim_{n\to \infty}\psi([\Sigma(s)]_{n-1})=\varphi(\Sigma(s)).
 \end{align*}
Claim~5 follows. 
 
 \smallskip
{\em Claim 6.} The map $\varphi\: J^\om \ra  S^2$ is continuous 
 and surjective.
 \smallskip
  
 Let $s\in J^\om $ and $n\in \N$. Then \eqref{eq:Cauest}
 shows that 
 \begin{equation}\label{eq:ssnclose}
 \varrho(\varphi(s), \psi([s]_n))\lesssim  \sum_{l=n}^\infty \Lambda^{-l}\lesssim 
 \Lambda^{-n}, 
 \end{equation}
 where $C(\lesssim)$ is independent of $n$ and $s$.
 Hence if $s,s'\in J^\omega$
 and $[s]_n=[s']_n$, then 
 \begin{equation*}
   \varrho(\varphi(s), \varphi(s'))
   \lesssim 
   \Lambda^{-n}, 
 \end{equation*}
  where $C(\lesssim)$ is independent of $n$,  $s$, and $s'$.
 The continuity of $\varphi$ follows from this; indeed, if $s$ and $s'$ are close in $J^\om$, then $[s]_n=[s']_n$ for some large $n$, and so the image points $\varphi(s)$ and $\varphi(s')$ are close in $S^2$. 
 
Since $J^\om $ is compact, the continuity of $\varphi$ implies that 
the image $\varphi(J^\om )$ is also compact and hence closed in $S^2$. The surjectivity of  $\varphi$ will follow, if we can show that $\varphi$ has a dense image in $S^2$.

To see this, let $x\in S^2$ and $n\in \N$ be arbitrary.
Then by Claim~1 we can find a point $y\in f^{-n}(p)$ with 
$\varrho(x,y)\lesssim \Lambda^{-n},   $ where $C(\lesssim)$ is independent of $x$ and $n$. 
Moreover, by Claim 3 there exists a word $w\in J^n$ with 
$\psi(w)=y$. Pick $s\in J^\om$ such that $[s]_n=w$. Then by
\eqref{eq:ssnclose} we have 
$$\varrho(x,\varphi( s))\le \varrho(x,y)+\varrho(y, \varphi(s))=\varrho(x,y)+\varrho(\psi([s]_n),\varphi(s))\lesssim \Lambda^{-n}, $$
where $C(\lesssim)$ is independent  of the choices.
Hence 
$$\sup_{x\in S^2}\dist(x, \varphi(J^\om))\lesssim \Lambda^{-n}$$
for all $n$, where $C(\lesssim)$ is independent of $n$. This shows that  
  $\varphi(J^\om)$ is dense in $S^2$. Claim~6 follows. 
  
  \smallskip 
  The theorem now follows from Claim 5 and Claim 6.
 \end{proof}

The  procedure that we employed  to code 
the elements in $f^{-n}(p)$ by words of length $n$ and the points in $S^2$ by infinite words  is well known (see \cite[Section~5.2]{Ne}, for example).  Note that an expanding  Thurston map may have 
 periodic critical points.
Then there are
points in $S^2$ that are coded by an uncountable number of
sequences in $J^\omega$. 

\begin{proof}[Proof of Corollary~\ref{cor:perdense}] We use the notation and  setup of  the proof of Theorem~\ref{thm:expThfactor}.  

It suffices to show that if $x\in S^2$ and $n\in \N$ are arbitrary, then there exists a point $z\in S^2$ with $f^n(z)=z$ and 
$\varrho(x,z)\lesssim \Lambda^{-n}$. Here and in the following,  $C(\lesssim)$ 
is  independent of $x$ and $n$. 

 To find such a point $z$, we  apply Claim~1 in the proof of 
Theorem~\ref{thm:expThfactor} and  conclude that  there exists $y\in f^{-n}(p)$ with 
$\varphi(x,y)\lesssim \Lambda^{-n}$. By Claim~3 in this proof there exists a word 
$w\in J^*$ of length $n$ such that $\psi(w)=y$. Let $s$ be the unique sequence obtained by periodic repetition of the letters in $w$, i.e., $s\in J^\om$ is the unique sequence with $[s]_n=w$ and $\Sigma^n(s)=s$. 
Put $z\coloneqq \varphi(s)$. Then Claim~5 in the proof of Theorem~\ref{thm:expThfactor} implies 
$$ f^n(z)=f^n(\varphi(s))=\varphi(\Sigma^n(s))=\varphi(s)=z.$$
Moreover, by \eqref{eq:ssnclose} we have 
$$\varrho(y,z)=\varrho(\psi(w), \varphi(s))=\varrho(\psi([s]_n),  \varphi(s))\lesssim \Lambda^{-n}, $$
and so 
$$\varrho(x,z)\le \varrho(x,y)+\varrho(y,z)\lesssim \Lambda^{-n}. $$
The statement follows. \end{proof}

If in the previous argument we choose a constant sequence $s\in J^\om$ and set $z=\varphi(s)$, then $\Sigma(s)=s$, and so 
$$ f(z) =f(\varphi(s))=\varphi(\Sigma(s))=\varphi(s)=z. $$
This shows that every expanding Thurston map has a fixed point. 
 A more systematic investigation of fixed points and periodic points of expanding Thurston maps can be found in \cite{Li1}. 

\ifthenelse{\boolean{singlechapter}}{


%

\chapter{Tile graphs } \label{cha:Gromov}
 
An interesting feature of expanding Thurston maps is that they are linked to negatively curved spaces. Namely, if $f\: S^2\ra S^2$ is such a  map and $\CC\sub S^2$ is a Jordan curve 
with $\post(f)\sub \CC$, then one can use the associated cell decompositions 
to define an  infinite graph $\G=\G(f,\CC)$. The set of vertices of this graph is given by the  collection  of  tiles on all levels, where it is convenient to add $X^{-1}\coloneqq S^2$ as a tile of level $-1$ and basepoint of the graph. One connects two vertices 
by an edge if the corresponding tiles have non-empty intersection and their levels differ by at most $1$. 
  We will study the properties of this {\em tile graph}  in the
  present chapter. The main results are based on work by Q.~Yin
  (see \cite{Qian15b}).

  \begin{theorem} \label{thm:tileGrom} Let $f\: S^2\ra S^2$ be an expanding Thurston map, and $\CC\sub S^2$ be a Jordan curve with $\post(f)\sub \CC$. Then the  associated tile graph $\G(f,\CC)$ is Gromov hyperbolic. 
  \end{theorem} 
\index{Gromov!hyperbolic}

For the boundary at infinity $\partial_\infty \G$ we have a natural 
identification $\partial_\infty \G \cong S^2$.  By this identification the class of visual metrics in the sense of Thurston maps (see Chapter~\ref{cha:visual-metrics})  and in the sense of Gromov hyperbolic spaces (see Section~\ref{sec:Grhyp})  are the same. 

\begin{theorem} \label{thm:visualGrTh} Let $f\: S^2\ra S^2$ be an
  expanding Thurston map, $\CC\sub S^2$ a Jordan curve with
  $\post(f)\sub \CC$, and $\G=\G(f,\CC)$ be the associated tile
  graph. Then $\partial_\infty \G$ can naturally be identified with
  $S^2$. Under this identification, a metric  on
  $\partial_\infty \G\cong S^2$ is visual in the sense of Gromov
  hyperbolic spaces if and only if it is visual in the sense of
  expanding Thurston maps.
\end{theorem} 
\index{visual metric}
\index{visual metric!on $\geo X$}  

Under the identification of $S^2$ with $\partial_\infty\G$ as in the
previous theorem, the number $m_{f,\CC}$ (see
Definition~\ref{def:mxy}) is the Gromov product $(x\cdot y)$ (with basepoint $X^{-1}$) up to
some additive constant.

\index{Gromov!product}
\index{m@$m_{f,\CC}$} 
\begin{lemma}
  \label{lem:m_Gromov}
  In the setting of Theorem~\ref{thm:visualGrTh} there is a constant
  $c\geq0$ such that
  \begin{equation*} 
    m_{f,\CC}(x,y)-c \le (x\cdot y)\le m_{f,\CC}(x,y)+c, 
  \end{equation*}
  for all $x,y\in S^2$. 
\end{lemma}

In \cite{HP}  Ha\"{i}ssinsky and Pilgrim also considered 
 the sphere $S^2$ as the boundary at infinity of a suitable Gromov hyperbolic space 
very similar to the setting  of Theorem~\ref{thm:visualGrTh}. 

 An obvious question is how the graphs $\G(f, \CC)$ and $\G(f,\widetilde  \CC)$ are related for different Jordan curves $\CC,\widetilde  \CC\sub S^2$ containing $\post(f)$. For a Cayley graph of a group a change of the generating set leads to quasi-isometric Cayley graphs. So in the context of expanding Thurston maps one may expect a similar result. Actually, a stronger statement is true: the graphs  $\G(f, \CC)$ and $\G(f, \widetilde \CC)$ are even rough-isometric (see Section~\ref{sec:Grhyp}).

 \begin{theorem} \label{thm:roughisom} Let $f\: S^2\ra S^2$ be an expanding Thurston map, and $\CC, \CC'\sub S^2$ be  Jordan curves with $\post(f)\sub \CC, \CC'$.
 Then the graphs $\G(f,\CC)$ and $\G(f,\CC')$ are rough-isometric. 
   \end{theorem}

Throughout this chapter,  $f\: S^2\ra S^2$ will be an expanding Thurston map, and $\CC\sub S^2$ a Jordan curve with  $\post(f)\sub \CC$. We consider tiles in the cell decompositions $\DD^n(f,\CC)$, $n\in \N_0$, and add $X^{-1}\coloneqq S^2$ as a tile of  level $-1$. Let $\X'$ be the collection of  tiles on all levels $n\in \N_0\cup\{-1\}$. In $\X'$ we  consider tiles as different if 
 their levels are different even if the underlying sets of the tiles are the same. If $X\in 
 \X'$, we denote by 
 \index{laaa@$\ell(X)$}
 \begin{equation}
   \label{eq:def_lX}
   \ell(X)\in \N_0\cup\{-1\}
 \end{equation}
the level of the tile $X$; so $X$ is an $\ell(X)$-tile. 
 
As discussed above,  we define the 
{\em tile graph}\index{tile!graph}\index{G(f,c)@$\G(f,\CC)$}
$\G=\G(f,\CC)$ of $f$ with respect to $\CC$ as follows. The set of vertices of $\G$ is equal to  the set 
 $\X'$ of all tiles. Moreover, two distinct vertices given by a $k$-tile $X^k$ 
 and an $n$-tile $X^n$ are joined by an edge precisely if 
 \begin{equation}
   \label{eq:tileedgerel}
   \abs{k-n}\le 1 \text { and } X^k\cap X^n\ne \emptyset. 
 \end{equation}
 So we join two vertices if the corresponding tiles intersect and
 their levels differ by at most $1$. The graph $\G$ is a
 $1$-dimensional cell complex where each edge is identified with
 an interval of length $1$. Then the graph $\G$ is connected.
 
 Indeed, each point contained in an edge of  $\G$ can be joined to a vertex  of  $\G$,  given by an $n$-tile $X^n$. We can join  $X^n$  to $X^{-1}=S^2\in \G$ as follows. Pick a point $p\in X^n$, and
 for $i=0, \dots, n-1$ let $X^i$ be an $i$-tile with $p\in X^i$. Then
 in the vertex sequence $X^n, X^{n-1}, \dots , X^{-1}$ two consecutive
 elements are joined by an edge, because the levels of the tiles
 differ by $1$ and all tiles contain $p$ and hence have non-empty
 intersection.  So there exists a path joining $X^n$ and $X^{-1}$ in
 $\G$ as desired.
 
Since $\G$ is connected,  this graph carries a unique path metric so that each edge is isometric to the unit interval.  If $X$ and $Y$ are vertices in 
 $\G$, i.e., tiles in $\X'$, we denote by $\abs{X-Y}$ the distance
 of $X$ and $Y$ in $\G$, and call this quantity   the {\em combinatorial distance} of the tiles. 
 By definition of the metric in 
 $\G$ it is clear that $\abs{X-Y}$ is equal to the minimal number $n\in \N_0$ such that there exist tiles $X_0=X, X_1, \dots, X_n=Y$ in $\X'$ satisfying 
 \begin{equation}
   \label{eq:combdist}
   \abs{\ell(X_{i-1})-\ell(X_{i})} \le 1 
   \text{ and } 
   X_{i-1}\cap X_i\ne \emptyset
 \end{equation}
 for $i=1, \dots, n$. 
 
 Note that 
 \begin{equation*}
    \abs{X-Y}\ge \abs{\ell(X)-\ell(Y)}.
 \end{equation*}
 Moreover, if $X\cap Y\ne \emptyset$, then 
 a simple argument similar to  the one we have just used to show connectedness of  $\G$ gives that 
 $$ |X-Y|\le |\ell (X)-\ell(Y)|+1. $$
We pick $X^{-1}=S^2$ as the basepoint in $\G$. Note that the
combinatorial distance of a vertex $X\in \G$ to $X^{-1}$ is
\begin{equation*}
  \abs{X-X^{-1}}= \ell(X)+1. 
\end{equation*}
We denote by $(X\cdot Y)$ the Gromov product (see
\eqref{eq:def_Grpr}) of two vertices $X,Y\in
 \G$ with respect to the basepoint $X^{-1}$; 
 so\index{Gromov!product} 
 \begin{align}
   \label{eq:gromov_tileG}
   (X\cdot Y)&=\frac12 (\abs{X-X^{-1}}+\abs{Y-X^{-1}}-\abs{X-Y})
   \\
   \notag
             &=1+\frac 12 (\ell(X)+\ell(Y)-|X-Y|). 
 \end{align}

 We want to  show that the graph $\G$ equipped with its path metric is a Gromov hyperbolic space. Since every point in $\G$
 has distance $\le 1/2$ to a vertex, the set $ \X'$ of vertices in 
 $\G$ is cobounded in $\G$. Hence it is enough to  consider 
 the space of vertices $ \X'$ equipped with the metric given by the combinatorial 
 distance of vertices. 
 The key for proving Gromov hyperbolicity of  $\G$ is to relate
 the Gromov product to visual metrics for $f$ as discussed in
 Chapter~\ref{cha:visual-metrics}. The basic idea goes back to a
 similar argument in \cite{BP}. Our presentation mostly follows
 \cite{Qian15b}.

 For the rest of this chapter we  pick a fixed visual metric $\varrho$ for $f$ on $S^2$. Metric notions on $S^2$ will refer to $\varrho$ unless otherwise stated. If $\Lambda>1$ is the expansion factor of $\varrho$, then by Proposition~\ref{lem:expoexp} 
  \begin{equation}\label{eq:tilesandmetric}
 \diam(X) \asymp \Lambda^{-\ell (X)}
 \end{equation}
 for all $X\in \X'$, and 
 \begin{equation}\label{eq:tilesandmetric2}
 \dist(X,Y) \gtrsim  \Lambda^{-\ell (X)}
 \end{equation}
 for all $X,Y\in \X'$ with $X\cap Y=\emptyset$ and $\ell(X)=\ell(Y)$. 
In both inequalities   the implicit constants are independent of the tiles involved. 
  
We may view tiles in two different ways: as vertices in the graph
$\G$, or as subsets of the sphere $S^2$. The following lemma
provides the key for  relating  these two viewpoints.
 
\begin{lemma}\label{lem:basicGromov}
  For all tiles $X,Y\in
  \X'\sub \G$ we have  
   $$\Lambda^{-(X\cdot Y)} \asymp  \diam(X\cup Y),$$
   where $C(\asymp)$ is independent of $X$ and $Y$. 
  \end{lemma}
  
  See \cite[Lemma~2.2]{BP} for a similar statement in a different (but
  related) context.  
  
  \begin{proof} Let $X,Y\in  \X'$ be arbitrary, and set $n=|X-Y|\in \N_0$. 
    To prove the upper bound for $\diam(X\cup Y)$, we pick a tile chain (see
    Definition~\ref{def:chains}) 
    \begin{equation*}
      X_0=X, X_1, \dots, X_n=Y 
    \end{equation*}
    satisfying \eqref{eq:combdist} and realizing the
    combinatorial distance $n$ between $X$ and $Y$. Note that for the levels of the tiles in this chain we have 
    \begin{equation*}
      \ell(X_i)
      \ge 
      \max \{\ell(X)-i,\ell(Y)-(n-i)\}
    \end{equation*}
  for $i=0, \dots, n.$ 
  The minimum of the right hand side occurs for
  \begin{equation*}
    i=l
    \coloneqq 
    \lfloor(\ell(X)-\ell(Y)+n)/2\rfloor\in [0,n].
  \end{equation*}
  Using \eqref{eq:gromov_tileG} we have the estimates
  \begin{align*}
    &l-\ell(X) 
    \leq 
    \frac{1}{2}(n-\ell(X) - \ell(Y) ) 
    =
    -(X\cdot Y) +1 \text{ and}
    \\
    &  n - l -1-\ell(Y)  
    \leq 
    -(X\cdot Y) +1.
  \end{align*}
 So by \eqref{eq:tilesandmetric} 
  we have
%
  \begin{align*}
  \diam(X\cup Y)&\le  \sum_{i=0}^n\diam(X_i) \lesssim  
  \sum_{i=0}^n\Lambda^{-\ell(X_i)}
    \\
  &\lesssim \sum_{i=0}^l\Lambda^{i-\ell(X)}
    +\sum_{i=l+1}^n\Lambda^{(n-i)-\ell(Y)}
    \\
  &\lesssim  \Lambda^{-(X\cdot Y)} 
  \end{align*}
with  implicit constants  independent of $X$ and $Y$.  


To establish  the lower bound for $\diam(X\cup Y)$, let 
$m$ be the maxi\-mal 
 integer with $$-1\le  m\le \min\{\ell(X), \ell(Y)\}$$  such that  there exist $m$-tiles $X^m$ and $Y^m$ with 
 $X\cap X^m\ne \emptyset$, $Y\cap  Y^m\ne \emptyset$, and $ X^m\cap Y^m\ne \emptyset. $
 Then 
 $$|X-X^m|\le \ell(X)-m+1, $$ $$|Y-Y^m|\le \ell(Y)-m+1,$$   $$|X^m-Y^m|\le 1.$$
 This implies 
 \begin{align*}
  |X-Y|&\le  |X-X^m|+|X^m-Y^m|+|Y^m-Y|\\
  &\le \ell(X)+\ell(Y)-2m+3,
 \end{align*}
and so by \eqref{eq:gromov_tileG},
$$(X\cdot Y)=1+\frac12(\ell(X)+\ell(Y)-|X-Y|)\ge m-1/2.$$

Now if $m=\min\{\ell(X),\ell(Y)\}$, then by \eqref{eq:tilesandmetric} we have 
\begin{align*}
  \diam(X\cup Y)&\ge \max\{\diam(X), \diam(Y)\}\\ & \gtrsim  \max\{
\Lambda^{-\ell(X)}, \Lambda^{-\ell(Y)}\}  \\ &=
 \Lambda^{-m}\gtrsim  \Lambda^{-(X\cdot Y)}
 \end{align*}
with  implicit constants  independent of $X$ and $Y$. This gives the desired lower bound in this case. 

In the other case, where $m<\min\{\ell(X),\ell(Y)\}$, we pick
$(m+1)$-tiles $X^{m+1}$ and $Y^{m+1}$ with $X\cap X^{m+1}\ne
\emptyset$ and  $Y\cap Y^{m+1}\ne \emptyset$. 
Then there are  points 
$x\in X\cap X^{m+1}$ and $y\in Y\cap Y^{m+1}$. Moreover, by definition of $m$ we  have $X^{m+1}\cap Y^{m+1}=\emptyset$. Hence 
by \eqref{eq:tilesandmetric2} we have  
\begin{align*}
  \diam(X\cup Y)&\ge \varrho(x,y)\,\ge  \,\dist(X^{m+1}, Y^{m+1}) 
  \\ &\gtrsim \Lambda^{-(m+1)}\gtrsim \Lambda^{-(X \cdot Y)}
\end{align*}
with  implicit constants  independent of $X$ and $Y$.
So we get the desired lower bound also in this case. 
\end{proof} 

The following consequence of the previous lemma relates sequences
converging to infinity in $\G$ with points in the sphere
$S^2$. 

\pagebreak

\index{convergence to infinity}
\begin{lemma}
  \label{lem:seq_tiles_Gromov}
  Let $\{X_i\}$ be a sequence of points (i.e., tiles) in $\X'$. 
  Then  the following statements are true: 
  \begin{enumerate}
  \item 
\label{item:Xiinfty_Xip}
    $\{X_i\}$ converges to infinity in $\G$ if and only if
    there is a unique point $p\in S^2$ such that $X_i\to \{p\}$ as $i\to \infty$ in the sense of  Hausdorff convergence  on 
     $S^2$.
  \item 
\label{item:Xi_equiv_Yi}
    Another sequence $\{Y_i\}$ in $\X'$ that converges to
    infinity in $\G$ is equivalent to $\{X_i\}$ if and only if
  the sequences  Hausdorff  converge to the same 
    singleton set $\{p\}\sub S^2$.
  \end{enumerate}
\end{lemma}
 
For the definition of  Hausdorff convergence see the end of Section~\ref{sec:QCgeom}. Recall from \eqref{eq:defxi_infty} that a sequence $\{X_i\}$ in $\X'$
converges to infinity if and only if
\begin{equation}  \label{eq:limXX}
  \lim_{i,j\to \infty}(X_i\cdot X_j)=\infty
\end{equation} 
and from \eqref{eq:defxy_eqiv} that a sequence $\{Y_i\}$ in $\X'$ that
converges to infinity is equivalent to $\{X_i\}$ if (and only if)
\begin{equation*}
  \lim_{i\to \infty}(X_i\cdot Y_i)=\infty.
\end{equation*}

\begin{proof}
  Note that \eqref{eq:limXX} is equivalent to 
\begin{equation}  \label{eq:diamXX}
\lim_{i,j\to \infty}\diam (X_i\cup  X_j)=0 
\end{equation}
by Lemma~\ref{lem:basicGromov}.
This is equivalent to $\diam(X_i)\to 0$ as $i\to \infty$ and that 
$\{X_i\}$  is a Cauchy sequence with respect to Hausdorff distance 
on $S^2$. This in turn happens if and only if there exists a unique
point $p\in S^2$ such that $X_i\to \{p\}$ as $i\to \infty$ in the
sense of Hausdorff convergence. Thus \ref{item:Xiinfty_Xip} holds.

\smallskip{}
Let $\{Y_i\}$ be another sequence in $\X'$ 
that converges to infinity. From Lemma~\ref{lem:basicGromov} we see
that $\{Y_i\}$ is equivalent to $\{X_i\}$ if and only if $\lim_{i\to
  \infty}\diam (X_i\cup  Y_i)=0$. 
This happens if and only if
$X_i$ and $Y_i$ Hausdorff converge (in $S^2$)  
to the same
singleton $\{p\}$ as $i\to \infty$. Thus \ref{item:Xi_equiv_Yi} also holds. 
\end{proof}

 \begin{proof}[Proof of Theorem~\ref{thm:tileGrom}]  We use notation as before. Since the set of vertices $ \X'$ is cobounded in $\G$, it suffices to show that 
 $ \X'$ equipped with the combinatorial distance is Gromov hyperbolic. 
 
 Now if $X,Y,Z\in  \X'$ are arbitrary, then 
 \begin{align*}
   \diam(X\cup Y)&\le \diam(X\cup Z)+\diam(Z\cup Y) 
   \\ 
                 &\le   2 \max\{ \diam(X\cup Z), \diam(Z\cup Y)\}.  
\end{align*}
 Invoking Lem\-ma~\ref{lem:basicGromov} and taking logarithms with base $\Lambda$ in the last  inequality, we obtain  
$$ (X \cdot Y)\ge \min \{ (X\cdot  Z), (Z\cdot Y)\} -\delta, $$
where $\delta\ge 0$ is a suitable constant independent of $X$,
$Y$, $Z$. Thus $\mathcal{G}$ is Gromov hyperbolic
(see~\eqref{def:Grprod}).  
\end{proof}

We are now ready to prove that our notion of visual metric on $S^2$
agrees with the standard one on $\partial_\infty\G$ under a
suitable identification. Recall from Section~\ref{sec:Grhyp} that
$\partial_\infty\G$ is defined to be the set of all equivalence
classes of sequences in $\G$ converging to infinity.
 
\begin{proof}[Proof of Theorem~\ref{thm:visualGrTh}] 
  The identification will be given by a bijection between
  $\partial_\infty\G$ and $S^2$. 
  Since $  \X'$ is cobounded in 
  $\G$, every point $x\in\partial_\infty \G$
  can be represented by a sequence of tiles $\{X_i\}$ in $\X'$ converging 
  to infinity. By
  Lemma~\ref{lem:seq_tiles_Gromov}~\ref{item:Xiinfty_Xip} there is a
  unique point $p\in S^2$ such that $X_i\to  \{p\}$ as $i\to \infty$ in the sense of  Hausdorff convergence. Any two sequences $\{X_i\},\{Y_i\}$
  representing $x$ converge to the same singleton $\{p\}$ by
  Lemma~\ref{lem:seq_tiles_Gromov}~\ref{item:Xi_equiv_Yi}. Thus the map
  \begin{equation*}
    \varphi\: \partial_\infty \G\to S^2 
    \;\text{ given by }\;
    \varphi(x) \coloneqq  p 
\end{equation*}
is well-defined. 
  
 The  map $\varphi$ is surjective. Indeed, let  $p\in S^2$  be arbitrary. For each $i\in \N$ we pick an $i$-tile $X_i\in   \X'$ such that $p\in X_i$. 
 Since $f$ is expanding, we have $\diam(X_i)\to 0$ as $i\to
 \infty$. Thus $X_i \to \{p\}$ as $i\to \infty$ (in the sense of Hausdorff convergence on $S^2$).  By
 Lemma~\ref{lem:seq_tiles_Gromov}~\ref{item:Xiinfty_Xip} the sequence 
 $\{X_i\}$ converges to infinity and the point
 $x\in \partial_\infty\G$ represented by $\{X_i\}$ is mapped to
 $p$ by $\varphi$. 

To show that $\varphi$ is injective, consider two points
$x,y\in \partial_\infty \G$ that are represented by sequences
$\{X_i\}$ and $\{Y_i\}$ in  $  \X'$ converging to infinity. By
Lemma~\ref{lem:seq_tiles_Gromov}~\ref{item:Xi_equiv_Yi}
they converge to the same singleton set if and only if they are
equivalent. Thus $\varphi(x)= \varphi(y)$ if and only if $x=y$. Thus
$\varphi$ is injective.  
 
Having proved that $\varphi$ is bijective, we ignore the
original distinction between points in $\partial_\infty \G$ and
in $S^2$, and identify $\partial_\infty \G\cong S^2$ by the map
$\varphi$. 

To prove  the second part of the statement, we first consider the visual metric 
$\varrho$ (in the sense of Thurston maps)  for $f$  fixed earlier. 
Let $x$ and $y$ be arbitrary points in $S^2$ and 
let $\{X_i\}$ and $\{Y_i\}$ be two sequences in  $ \X'$
representing them in $\partial_\infty \G$,  respectively. As we have seen, 
 this means 
$X_i\to\{x\}$ and $Y_i\to \{y\}$ as $i\to \infty$ in the sense of Hausdorff convergence on $S^2$. Thus 
\begin{equation*}
   \diam(X_i\cup Y_i)\to \varrho(x,y)
\end{equation*}
as $i\to \infty $.

Recall the definition of the Gromov
product $(x\cdot y)$ for $x,y\in \partial_\infty\G\cong S^2$ as given in  \eqref{def:Grprodinfty}.  Here  we choose $X^{-1}=S^2$ as 
the basepoint in $\G$.
By \eqref{def:Grprodinfty2} there exists a constant $k\ge 0$
independent of $x$ and $y$, and of the choice of the sequences
$\{X_i\}$ and $\{Y_i\}$ such that 
$$\liminf_{i\to \infty} (X_i\cdot Y_i)-k \le (x\cdot y)\le 
\liminf_{i\to \infty} (X_i\cdot Y_i). $$ 
If we combine the previous two estimates with Lemma~\ref{lem:basicGromov}, then we  conclude that 
$ \varrho(x,y)\asymp \Lambda^{-(x\cdot y)}$. 

Since $\varrho$ is a visual metric for $f$, by Proposition~\ref{prop:visualsummary}~\ref{item:visCC} we  have $\varrho(x,y)\asymp \Lambda^{-m(x,y)}$,
where $m=m_{f,\CC}$ is as in Definition~\ref{def:mxy}. It follows  that 
\begin{equation} 
  \label{eq:m=Grprod} 
  m(x,y)-c \le (x\cdot y)\le m(x,y)+c, 
\end{equation} 
where $c \ge 0$ is independent of $ x$ and $y$. 

In the definition of
visual metrics in the sense of Gromov hyperbolic spaces 
\eqref{vismetric} we may choose any basepoint for the Gromov product (up to an adjustment of
the multiplicative constant). Similarly, by
Proposition~\ref{prop:visualsummary}~\ref{item:visCC} it is not a restriction to
use our  given curve $\CC$ in Definition~\ref{def:visual} of visual metrics in the sense
of expanding Thurston maps. 

Hence  \eqref{eq:m=Grprod} shows that any  metric $\widetilde{\varrho}$ on $\partial_\infty
\G\cong S^2$ is a visual metric in the sense of Gromov hyperbolic
spaces if and only if $\widetilde{\varrho}$ is a visual metric in the
sense of expanding Thurston maps. 
\end{proof}

Note that in \eqref{eq:m=Grprod} we have proved
Lemma~\ref{lem:m_Gromov}.

\begin{proof}[Proof of Theorem~\ref{thm:roughisom}] It suffices to
  find a rough-isometry between the set of vertices in $\G=\G(f,\CC)$
  and $\widetilde{\G}=\G(f,\widetilde \CC)$. As before, we denote the
  set of tiles for $(f,\CC)$ by $\X'$, and use the notation
  $\widetilde \X'$ for the set of tiles for $(f, \widetilde \CC)$
  (including $\widetilde X^{-1}=X^{-1}=S^2$).
   
   For each  tile $X\in  \X'$ we pick a tile $\widetilde X\in \widetilde \CC$ of the same level (i.e., $\ell(X)=\ell(\widetilde X)$) with $X\cap \widetilde X\ne \emptyset$. 
   This  assignment $X \mapsto \widetilde X$ gives a level-preserving map 
   $\psi\: \X'\ra \widetilde \X'$. We claim that $\psi$ is a rough-isometry between $\X'$ and $\widetilde \X'$, where the  spaces are equipped with their respective combinatorial distances.
   
   To see this, let $X,Y\in \X'$ be arbitrary, and consider $\widetilde X\coloneqq \psi(X)$ and 
   $\widetilde Y\coloneqq \psi(Y)$. Since $\psi$ preserves levels of tiles, we have
   $$ \diam(X)\asymp \Lambda^{-\ell(X)} =\Lambda^{-\ell(\widetilde X)}   
   \asymp  \diam(\widetilde X),$$
   and similarly $\diam(Y)\asymp \diam(\widetilde Y)$.
   Hence 
   \begin{align*}
   \diam(X\cup Y)&\le \diam (X) + \diam (\widetilde X \cup \widetilde Y) +\diam(Y)\\
   &\lesssim   \diam (\widetilde X \cup \widetilde Y), 
   \end{align*} 
   and the same argument gives  $\diam (\widetilde X \cup \widetilde Y)\lesssim  \diam(X\cup Y)$. 
   In all the previous relations the implicit multiplicative constants are independent of 
   $X$ and $Y$. 
   Lemma~\ref{lem:basicGromov} implies that there exists a constant $c\ge 0$ independent of $X$ and $Y$ such that 
   $$ (X\cdot Y)-c\le (\widetilde X\cdot \widetilde Y)\le (X\cdot Y)+c.$$
   Here Gromov products are in $\X'$ and $\widetilde \X'$, respectively,  with respect to the basepoint $X^{-1}=\widetilde{X}^{-1}=S^2$. Since $\ell(X)=\ell(\widetilde X)$ and $\ell(Y)=
   \ell(\widetilde Y)$, based on   \eqref{eq:gromov_tileG} we deduce the inequality 
   \begin{equation*}
        \abs{X- Y}-k 
        \le 
        \abs{\widetilde X-\widetilde Y}
        \le 
        \abs{X-Y}+k 
   \end{equation*}
   for combinatorial distances.  
   Here $k\coloneqq 2c$ is independent of $X$ and $Y$.    
   
  This is the first condition \eqref{eq:qisom1} (with $\lambda=1$) for $\psi$ to be a rough-isometry. It remains to show that 
  $\psi(\X')$ is cobounded in $\widetilde \X'$. To verify this, let $\widetilde Y\in \widetilde \X'$ be arbitrary. Pick a tile $X\in \X'$ with  $\ell(X)=\ell(\widetilde Y)$ and $X\cap \widetilde Y\ne \emptyset$. Define $\widetilde X\coloneqq \psi(X)$. It suffices to produce a uniform upper bound for the combinatorial distance $|\widetilde X- \widetilde Y|$
  of $\widetilde X$ and $\widetilde Y$ in $\widetilde \X'$ independent  of $\widetilde Y$. 
  Now  $$ \diam (X) \asymp \diam (\widetilde X)\asymp  \diam(\widetilde Y)\asymp \Lambda^{-\ell(\widetilde Y)},$$ and so 
    \begin{equation*}
   \diam(\widetilde X\cup \widetilde Y)\le  \diam (\widetilde X) + \diam (X) +
   \diam(\widetilde Y) \lesssim \Lambda^{-\ell(\widetilde Y)}, 
   \end{equation*} 
where again all implicit multiplicative constants are independent of the choice of the tiles. So by Lemma~\ref{lem:basicGromov} we have 
$$ (\widetilde X\cdot \widetilde Y)\ge \ell(\widetilde Y)-c',$$ 
where $c'\ge 0$ is independent of the choices. 
Since $\ell(\widetilde X)=\ell( \widetilde Y)$, we conclude
$$ |\widetilde X- \widetilde Y|=2+2\ell(\tilde Y)-2(\widetilde X\cdot \widetilde Y)\le k'\coloneqq 2+2c', $$
which gives the desired uniform bound. 
\end{proof}

\begin{rem}\label{rem:natrouiso} 
  The rough-isometry $\psi$ between the graphs $ \G$ and
  $\widetilde{\G}$ constructed in the previous proof is compatible
  with the identifications $\geo \G\cong S^2$ and $\geo
  \widetilde{\G}\cong S^2$.  Indeed, let $p\in S^2$ be
  arbitrary. Viewed as an element 
  in $\partial_\infty\G$, the point $p$ is represented by a sequence
  $\{X_i\}$ in $\X'$ such that $X_i\to \{p\}$ as $i\to \infty$ in the sense of 
  Hausdorff convergence  (see
  Lemma~\ref{lem:seq_tiles_Gromov}). If $\widetilde X_i\coloneqq \psi(X_i)$
  for $i\in \N$, then $\{\widetilde X_i\}$ is a sequence of tiles in
  $\widetilde \X'$.  By definition of $\psi$ the levels of $X_i$ and
  $\widetilde X_i$ are the same, and so  $\diam(\widetilde{X}_i)\to 0$ as
  $i\to \infty$. In addition,  $X_i\cap \widetilde X_i\ne \emptyset$ for
  $i\in \N$. This implies that $\widetilde X_i\to \{p\}$ as $i\to
  \infty$.  So if $\{X_i\}$ represents the point $p\in S^2$ under the
  identification $\geo \G\cong S^2$, then the image sequence $\{\widetilde
  X_i\}$ under $\psi$ also represents the point $p$ under the
  identification $\geo \widetilde{\G}\cong S^2$.
\end{rem}
\ifthenelse{\boolean{singlechapter}}{

%
%


\chapter{Isotopies} 
\label{cha:iso}
 
In this chapter we consider various questions related to isotopies. We first revisit the notion of Thurston equivalence,
which is defined in terms of certain isotopies. Then we investigate   when two Jordan curves in $S^2$ are isotopic
relative to a finite set of points. This is  in
preparation for  results about the existence of invariant Jordan curves  for expanding Thurston maps (see
Chapter~\ref{cha:constructc}).

Recall that two Thurston maps $f\: S^2\ra S^2$ and
$g\: \widehat S^2\ra \widehat S^2$ on $2$-spheres $S^2$ and
$\widehat S^2$ are (Thurston) equivalent (see
Definition~\ref{def:Thequiv}) if there exist homeomorphisms
$h_0,h_1\:S^2\ra \widehat S^2 $ that are isotopic rel.\
$\post(f)$ 
and satisfy  
$ h_0\circ f = g\circ h_1$.  We then
have the commutative diagram:
\begin{equation}\label{Thequiv2}
  \xymatrix{
    S^2 \ar[r]^{h_1} \ar[d]_f & \widehat S^2 \ar[d]^g
    \\
    S^2 \ar[r]^{h_0} & \widehat S^2\rlap{.}
  }
\end{equation}
The maps $f$ and $g$ are topologically conjugate if there exists
a homeomorphism $h\: S^2\ra \widehat S^2 $ such that
$h\circ f = g\circ h$.

Obviously, Thurston equivalence is a weaker notion than
topological conjugacy. However, two expanding Thurston maps $f$ and $g$ are
Thurston equivalent if and only if they are topologically
conjugate.

\begin{theorem}
  [Thurston equivalence and topological conjugacy]
  \label{thm:exppromequiv}
  \index{Thurston!equivalent} 
  Suppose $f\: S^2\to S^2$ and $g\:\widehat S^2\to \widehat S^2$
  are expanding Thurs\-ton maps that are Thurston equivalent. 
  Then they are
  topologically  conjugate.\index{topologically conjugate} 
  
  More precisely,     
  if we have a Thurston equivalence between $f$ and $g$ as in
  \eqref{Thequiv2}, then  there exists a homeomorphism  $h\:S^2\ra
  \widehat S^2$ such that $h$ is isotopic to $h_1$
  rel.~$f^{-1}(\post(f))$ and satisfies $h\circ f=g\circ h$. 
\end{theorem}
Since $h_0$ and $h_1$ are
isotopic rel.\ $\post(f)$ and $\post(f)\sub f^{-1}(\post(f))$,   this implies that $h$ is also isotopic
to $h_0$ rel.\ $\post(f)$.

A statement very similar to the theorem above was
proved by Kameyama \cite{Ka03a}. Since his notion of
``expanding'' is different from ours, we will present the details
of the proof.

Theorem~\ref{thm:exppromequiv} will be shown in
Section~\ref{sec:topol-interl-1}. To this end, we will repeatedly
lift the isotopy between $h_0$ and $h_1$ in \eqref{Thequiv2}. The
relevant result about the existence of such lifts is established in 
Proposition~\ref{prop:isotoplift}. Since the maps  $f$ and $g$ in
the above statement are expanding, the ``tracks'' of the lifted
isotopies under the $n$-th iterates will shrink exponentially with respect to a visual
metric as $n\to \infty$ (see Lemma~\ref{lem:exp_shrink}). We will
obtain the desired conjugacy between $f$ and $g$ essentially by concatenating these lifts. The basic idea  of this argument is well known in dynamics (see \cite{Sh69}, for example).

In Sections~\ref{sec:isotopy-rel.-postf} 
and~\ref{sec:graphs} we present some technical results on 
isotopies of Jordan curves. The most important result obtained here is
Lemma~\ref{lem:isorelP}, which gives a criterion when a Jordan curve can be isotoped into the $1$-skeleton of a given cell decomposition of $S^2$.  This  will be  a crucial ingredient  in the
proof of Theorem~\ref{thm:main}. 

Before we go into the details, we first fix some notation and  terminology related to homotopies and isotopies that will be used 
throughout this chapter (for the basic definitions see Section~\ref{sec:thurston-equivalence}).  We denote by $I\coloneqq [0,1]$ the unit interval. 
If  $X$ and $Y$ are topological spaces, and   $H\: X\times I \ra
Y$ is a 
homotopy\index{homotopy} 
between  $X$ and $Y$, then, as usual,   
$H_t\coloneqq H(\cdot, t)\: X\ra Y$ for $t\in I$ denotes  the time-$t$ map of the homotopy.

Conversely, when we say that a family $H_t$ 
of continuous maps from $X$ into $Y$ is a homotopy 
between $X$ and $Y$, it is understood that $t$ is a variable in $I$ and that the map $(x,t)\in X\times I\mapsto H_t(x)$ is a
homotopy. 
This is a slightly imprecise, but convenient way of
expression. Such a family $H_t$ is an 
isotopy\index{isotopy}
between $X$ and $Y$ if each map $H_t$ is a homeomorphism between $X$ and $Y$.

\section[Equivalent expanding maps are conjugate]{Equivalent expanding Thurston maps are conjugate}
\label{sec:topol-interl-1}

In preparation for  the proof of Theorem~\ref{thm:exppromequiv}, we
first   record a simple lemma about preimages of sets.  
\begin{lemma}
  \label{lem:lifts_inverses} 
  Let $f\:X\ra X$ and $g\:Y\ra  Y$ be maps 
   defined on some sets $X$ and $Y$, and $h,\widetilde h\:X\ra Y$  be  bijections 
  with  $g\circ \widetilde h=h\circ f $.   Then for every  set $A\subset X$ we have 
  \begin{equation*}
    g^{-1}(h(A))= \widetilde h (f^{-1}(A)). 
  \end{equation*}
   
\end{lemma}

\begin{proof} 
  Since $h, \widetilde h$ are bijections, $g\circ \widetilde h= h \circ f$
  implies $h^{-1}\circ g = f \circ \widetilde h^{-1}$. Thus
  \begin{equation*}
    g^{-1}(h(A))
    = 
    \left(h^{-1}\circ g\right)^{-1}(A)    
    = 
    \big(f\circ \widetilde h^{-1}\big )^{-1}(A) 
    = 
    \widetilde h(f^{-1}(A)),
  \end{equation*}
  as desired.
\end{proof}

We now turn to  lifts of isotopies by Thurston maps (see 
\cite[Lemma~4.3]{Ka03a} for a similar statement).

\begin{prop}[Lifts of isotopies by Thurston maps]\label{prop:isotoplift}\index{isotopy!lift}\index{Thurston map!lift of isotopy by}
  Suppose  $f\: S^2\to S^2$ and $g\:\widehat S^2\ra \widehat S^2$ are   Thurston maps, and $h_0,\widetilde h_0\colon S^2\to\widehat S^2$ are  homeomorphisms  
  such that  $h_0|\post(f)=\widetilde h_0|\post(f)$ and    $ g\circ \widetilde h_0= h_0\circ f $. Let  $H\:S^2\times I\ra \widehat  S^2$ be  an isotopy 
   rel.\   $\post(f)$ with $H_0=h_0$. 
   
   Then the isotopy $H$ uniquely lifts to an isotopy $\widetilde{H}\:S^2\times I\ra \widehat S^2$ rel.\  $f^{-1}(\post(f))$ such that $\widetilde{H}_0=\widetilde h_0$ and $g\circ \widetilde{H}_t=H_t\circ f $ for all 
   $t\in I$. \end{prop}
   
 So if we  set ${h}_1\coloneqq {H}_{1}$ and $\widetilde{h}_1\coloneqq \widetilde{H}_{1}$, then  we obtain the following commutative diagram:
    \begin{equation*}
      \xymatrix{
        S^2 \ar[rr]^{\widetilde{H}\colon\widetilde{h}_0\simeq \widetilde{h}_1} \ar[dd]_f
        & & \widehat S^2 \ar[dd]^g
        \\ & & &
        \\
        S^2 \ar[rr]^{ H\colon h_0\simeq h_1}
        & &\widehat  S^2\rlap{.}
      }
    \end{equation*}
Here $H\: h_0\simeq h_1$, for example, indicates that $H$ is an isotopy with 
$H_0=h_0$ and $H_1=h_1$.

\begin{proof} We have 
   \begin{equation} 
    \label{postfg2}
    h_0(\post(f))=\widetilde h_0(\post(f))=\post(g) 
  \end{equation} as follows from the remark after Lemma~\ref{lem:T-eq_crit_post}.
  This implies that $$H_t(\post(f))=\post(g)$$ for all $t\in I$.
Therefore,
  $H_t|S^2\setminus \post(f)$ is an isotopy between
  $S^2\setminus \post(f)$ and $\widehat S^2\setminus \post(g)$.

Moreover, it follows from Lemma~\ref{lem:lifts_inverses} and \eqref{postfg2} that  
$$ \widetilde h_0(f^{-1}(\post(f)))=g^{-1} (h_0(\post(f)))=g^{-1}(\post(g)).$$ So  the map 
$\widetilde h_0|S^2\setminus f^{-1}(\post(f))$     can be considered as a lift 
of $$(H_0\circ f)|S^2\setminus f^{-1}(\post(f))=(h_0\circ f)|S^2\setminus f^{-1}(\post(f)) $$ by the 
(unbranched)  covering map (see Lemma~\ref{lem:brcovcov}) $$g\: \widehat S^2\setminus g^{-1}(\post(g))\ra \widehat S^2\setminus \post(g).$$
    By the usual homotopy lifting theorem for covering maps (see \cite[Proposition~1.30, p.~60]{Ha})   
the homotopy $(H_t\circ f)|S^2\setminus f^{-1}(\post(f))$ lifts to a unique  homotopy 
$\widetilde{H}\: (S^2\setminus f^{-1}(\post(f)))\times I \ra \widehat S^2\setminus g^{-1}(\post(g))$  such that
$$\widetilde{H}_0=\widetilde h_0| S^2\setminus f^{-1}(\post(f))$$ and 
$g\circ \widetilde{H}_t=H_t\circ f$ on $S^2\setminus f^{-1}(\post(f))$
for all $t\in I$.   

We claim that $\widetilde{H}$ has a   unique extension to a homotopy 
between $S^2$ and $\widehat S^2$. To see this,  let $q\in  f^{-1}(\post(f))$ be arbitrary, and set $p\coloneqq 
f(q)\in \post(f)$. Then there exists $\widehat p\in \post(g)$
such that  $H_t(p)=\widehat p$ for all $t\in I$. Since $g$ is a
branched covering map, we can find a small topological disk
$\widehat V\subset \widehat{S}^2$ containing 
$\widehat p$ such that each of the components of $g^{-1}(\widehat V)$ contains precisely one point in $g^{-1}(\widehat p)$. Since $H(\{p\}\times I)=\{\widehat p\}$ and  $H$ is uniformly continuous on
$S^2\times I$, we can choose a small neighborhood $V\subset S^2$ of $p$ such that 
$H(V\times I)\sub \widehat V$. Finally, we can find a small
topological disk $U \subset S^2$ containing
$q$ such that $f(U)\sub V$ and 
$$U'\coloneqq U\setminus\{q\}\sub S^2\setminus 
f^{-1}(\post(f)).$$ Then the set 
$U'\times I\sub (S^2\setminus 
f^{-1}(\post(f))) \times I $ is connected; so $\widetilde H(U'\times I)$ is also 
connected. Moreover,
 \begin{align*} g(\widetilde H(U'\times I))
&= \{ g(\widetilde H_t(u)): u\in U',\, t\in I\}   \\
&=  \{  H_t(f(u)): u\in U',\, t\in I \}    \sub H(V\times I) \sub \widehat  V.
\end{align*}    
Hence the connected set $\widetilde{H}(U'\times I)$ is contained in a unique component $\widehat U$ 
of $g^{-1}(\widehat V)$. By choice of $\widehat V$, this component contains a unique point $\widehat q\in g^{-1}(\widehat p)$. By making $\widehat V$  smaller if necessary, we can guarantee  that  the corresponding component 
$\widehat U$ of $g^{-1}(\widehat V)$ containing $\widehat q$ lies in an arbitrarily  small neighborhood of $\widehat q$
(this easily follows from the fact that  $g$ is a branched covering map).  Since  $
\widetilde H(U'\times I)\sub \widehat U$ as we have just seen, this implies that we can continuously extend $\widetilde H$ to $\{q\}\times I$ by setting 
$ \widetilde{H}(q,t)=\widehat q$ for $t\in I$. Since $q$ was an arbitrary element of  the finite set $f^{-1}(\post(f))$, we see that $\widetilde H$ has indeed an extension
to a homotopy between $S^2$ and $\widehat S^2$, also called  $\widetilde H$.
This extension is  unique, because $S^2\setminus f^{-1}(\post(f))$ is dense in $S^2$.
The previous argument also  shows that the extension  $\widetilde H$
is a homotopy rel.\ $f^{-1}(\post(f))$.  Moreover, again by density of $S^2\setminus f^{-1}(\post(f))$ in $S^2$ it is clear that on $S^2$ we have $\widetilde H_0=\widetilde h_0$ and  $g\circ \widetilde{H}_t=H_t\circ f$ for $t\in I$. 
We conclude that the isotopy $H$ can be lifted to a unique {\em homotopy} 
$\widetilde H$ with the desired properties.

To show that $\widetilde H$ is actually an {\em isotopy} between $S^2$ and 
$\widehat S^2$, we first note that the roles of $f$ and $g$ in the previous argument  can be reversed. So by lifting the isotopy $H^{-1}_t$,  we can find a unique  homotopy $\widetilde K_t$ between $\widehat S^2$ and $S^2$  such that $\widetilde K_0=\widetilde h^{-1}_0$
and  $f\circ \widetilde{K}_t=H^{-1}_t\circ g$ for $t\in I$. 
Then $ \widetilde K_0\circ  \widetilde H_0= \widetilde h_0^{-1}\circ  \widetilde h_0=\id_{S^2}$ and 
$$f\circ \widetilde K_t\circ  \widetilde H_t =  H^{-1}_t\circ g \circ  \widetilde H_t = H^{-1}_t\circ   H_t \circ f=f. $$ This implies that for each $p\in S^2$ the (continuous)  path $t\in I\mapsto  \widetilde K_t(  \widetilde H_t(p))$ starts at $p$ for $t=0$ and is contained in the finite set $f^{-1}(f(p))$.
Hence  $\widetilde K_t( \widetilde  H_t(p))=p$ for all $t\in I$ and $p\in S^2$, or equivalently,
 $ \widetilde K_t \circ   \widetilde H_t =\id_{S^2}$ for $t\in I$. A similar argument shows that 
 $ \widetilde H_t \circ  \widetilde K_t =\id_{\widehat S^2}$ for $t\in I$. It follows that for each $t\in I$, the map $ \widetilde H_t$ is a homeomorphism from $S^2$ 
 onto $\widehat S^2$ with the inverse $\widetilde K_t$.  So $\widetilde H_t$ is indeed the unique  isotopy with the  desired properties. 
\end{proof}

Note that if in the previous proposition $H$ is an isotopy relative to a set  $M\sub S^2$
 with $\post(f)\sub M$, then the lift 
$\widetilde H$ is an isotopy rel.\  $f^{-1}(M)$.
Indeed, if $p\in f^{-1}(M)$, then $f(p)\in M$ and so  
$$g(\widetilde H_t(p))=H_t(f(p))=h_0(f(p))\eqqcolon\widehat{q}$$
for all $t\in I$. Thus $t\mapsto \widetilde H_t(p)$ is a path 
contained in the finite set $g^{-1}(\widehat{q})$ and hence a
constant path.

If  the Thurston map $g$  in Proposition~\ref{prop:isotoplift} is  expanding, then repeated lifts are shrinking. This is made precise in the following lemma,
which will be of crucial importance in the proof of
Theorem~\ref{thm:exppromequiv}. 

\begin{lemma}[Exponential shrinking of  tracks of isotopies]
  \label{lem:exp_shrink}
Let $f\: S^2\to S^2$ and $g\:\widehat S^2\ra \widehat S^2$ be 
  Thurston maps, and  $H^n\:S^2\times I\ra \widehat S^2$ be   
  isotopies rel.\ $\post(f)$ satisfying
  $g\circ H^{n+1}_t=H^n_t\circ f $ for $n\in \N_0$ and $t\in I$.
 
  If $g$ is expanding and $\widehat S^2$ is equipped with a
  visual metric for $g$, then the tracks of the isotopies $H^n$
  shrink exponentially as $n\to \infty$. More precisely, if
  $\varrho$ is a visual metric for $g$ with expansion factor
  $\Lambda>1$, then there exists a constant $C\ge 1$ such
  that \begin{equation}\label{eq:expshrink} \sup_{x\in
      S^2}\diam_{\varrho}( \{H^n_t(x):t\in I\})\leq C \Lambda^{-n}
  \end{equation}
  for all $n\in \N_0$. 
\end{lemma}

\begin{proof}
  For all $n\in \N_0$ and $t\in I$ we have
  $g^n\circ H^n_t=H^0_t\circ f^n$; so for fixed $x\in S^2$ and
  $n\in \N_0$ the path $t\mapsto H^n_t(x)$ in $\widehat S^2$ is
  a lift of the path $t\mapsto H^0_t(f^n(x))$ by the map
  $g^n$. Recall that in the proof of Lemma
  \ref{lem:liftpathshrinks} we had to break up the path $\gamma$
  into $N$ pieces $\gamma_j$ so that $\diam_\varrho (\gamma_j)<\delta_0$
  (see also (\ref{defdelta})).  Since $H^0$ is uniformly
  continuous, we can choose the number $N$ uniformly for all the
  paths $t\mapsto H^0_t(y)$, $y\in S^2$. Since $g$ is expanding,
  Lemma~ \ref{lem:liftpathshrinks} then implies that
  \begin{equation*}
    \sup_{x\in S^2}\diam_{\varrho}( \{H^n_t(x):t\in I\})\lesssim
    \Lambda^{-n} 
  \end{equation*}
  for all $n\in \N$, where $C(\lesssim)$ is independent of $n$.
\end{proof}

After these preparations, we are ready for the proof of the main result in
this section. 
 
\begin{proof}[Proof of Theorem~\ref{thm:exppromequiv}] 
  Let $f\colon S^2\to S^2$ and $g\colon \widehat{S}^2\to
  \widehat{S}^2$ be two expanding Thurston maps that are equivalent. We
  want to prove that they are in fact topologically conjugate.
  The main idea of the proof is to lift a suitable initial
  isotopy repeatedly and use the fact that by Lemma
  \ref{lem:exp_shrink} the tracks of the isotopies shrink
  exponentially fast. The desired conjugacy is then obtained as a
  limit.
 
  By assumption there exists an isotopy $H^0_t$ between $S^2$ and
  $\widehat S^2$ rel.\ $\post(f)$ such that
  $h_0\circ f=g\circ h_1$, where $h_0=H^0_0$ and $h_1=H^0_1$.  By
  Proposition~\ref{prop:isotoplift} we can lift the isotopy
  $H_t^0$ between $h_0$ and $h_1$ to an isotopy $H^1_t$ rel.\
  $f^{-1}(\post(f))\supset \post(f)$ between $h_1$ and
  $h_2\coloneqq H^1_1$.  Note that the map $h_1$ plays two roles
  here: it is the endpoint $H^0_1$ of the initial isotopy
  $H_t^0$, and also a lift of $h_0$.
   
   Repeating this argument, we get homeomorphisms $h_n$ 
   and isotopies  $H^n_t$ between $S^2$ and $\widehat S^2$ rel.\
   $ \post(f)$ 
such that
$H_t^{n}\circ f =g \circ  H_t^{n+1}$,  
   $H^n_0=h_{n} $, and $H^n_1=h_{n+1} $ 
for all $n\in \N_0$ and $t\in I$. It follows from induction on $n$ and the remark after the proof of Proposition~\ref{prop:isotoplift} that $H^n_t$ is actually an isotopy rel.\ $f^{-n}(\post(f))$. 
   
   This yields an ``infinite  tower'' of isotopies as  in Figure~\ref{fig:tower_iso}.
   \begin{figure}
     \centering
  \begin{equation*}
  \xymatrix{
    {}\ar[d] & \overset{\vdots}{\phantom{X}} & {}\ar[d]
    \\ S^2 \ar[rr]^{H^2\colon h_2\simeq h_3} \ar[dd]_f
    & & \widehat S^2 \ar[dd]^g
    \\ & &
    \\ S^2 \ar[rr]^{H^1\colon h_1\simeq h_2} \ar[dd]_f
    & & \widehat S^2 \ar[dd]^g
    \\ & &
    \\ S^2 \ar[rr]^{H^0\colon h_0\simeq h_1}
    & & \widehat S^2\rlap{.}
  }
\end{equation*}
\caption{Tower of isotopies.}
\label{fig:tower_iso}
\end{figure}
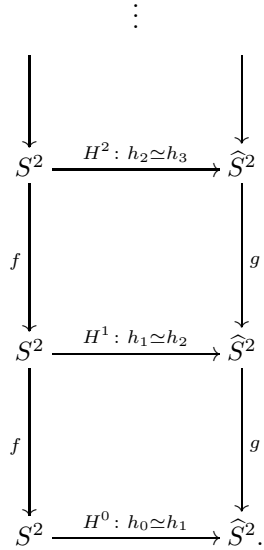
We want to show that for $n\to \infty$ the maps  $h_n$ converge to a homeomorphism 
$h_\infty$ that gives the  desired topological conjugacy  between $f$ and $g$. 

To see this, fix a visual metric $\varrho$ on $\widehat S^2$, and
assume that it has the expansion factor $\Lambda>1$.  Metric
concepts on $\widehat S^2$ will refer to this metric in the
following.  Since $g$ is expanding, Lemma~ \ref{lem:exp_shrink}
implies that
\begin{equation}\label{diamtracks}
  \sup_{x\in S^2}\diam( \{H^n_t(x):t\in I\})\lesssim \Lambda^{-n}
\end{equation}
for all $n\in \N$, where $C(\lesssim)$ is independent of $n$.
In particular, 
\begin{equation*}
  \dist(h_n,h_{n+1})
  \coloneqq 
  \sup_{x\in S^2} \varrho(h_n(x), h_{n+1}(x))
  \lesssim 
  \Lambda^{-n}
\end{equation*}
for all $n\in \N_0$, and so there is a continuous map 
$h_\infty\: S^2\ra \widehat S^2$ such that $h_n\to h_\infty$ uniformly on $S^2$ as $n\to \infty$.  Since 
$h_{n-1}\circ f= g\circ h_n$, we have
$h_{\infty}\circ f=g\circ h_\infty$. 

The map $h_\infty$ is a homeomorphism. To prove this,  we repeat the argument where we interchange the roles of $f$ and $g$. More precisely, we consider the isotopy $(H_t^0)^{-1}$ between $h_0^{-1}$ and $h_1^{-1}$.  The corresponding tower of repeated lifts of this initial isotopy is given by the isotopies $(H^n_t)^{-1}$ between 
 $h_{n}^{-1}$
and $h_{n+1}^{-1}$.  By the argument in the first part of the proof we see that  the maps $h_n^{-1}$ converge to a continuous map 
$k_\infty\: \widehat S^2\ra S^2$ uniformly on $\widehat S^2$ as $n\to \infty$. 
By uniform convergence we have 
$(k_\infty\circ h_\infty)(x)=\lim_{n\to \infty}( h_n^{-1}\circ h_n)(x)=x$ 
for all $x\in S^2$. Hence $k_\infty\circ h_\infty=\text{id}_{S^2}$. Similarly,
$h_\infty\circ k_\infty=\text{id}_{\widehat S^2}$, and so $k_\infty$ is a continuous inverse 
of $h_\infty$. Hence $h_\infty$ is a homeomorphism.

The conjugating map $h=h_\infty$ is isotopic to 
$h_1$ rel.~$f^{-1}(\post(f))$. To see this, we will define an isotopy rel.\  $f^{-1}(\post(f))$
that is obtained by concatenating (with suitable time change) the isotopies $H^1, H^2, \dots$ 
and take $h=h_\infty$ as the endpoint at time $t=1$. The precise
definition is as follows.
We break up the unit interval into intervals
\begin{equation*}
  I=[0,1]= \left[0,\tfrac{1}{2}\right]\cup 
  \left[\tfrac{1}{2},\tfrac{3}{4}\right] \cup \dots \cup 
  \left[1-2^{-n},1 -2^{-n-1}\right]
  \cup \dots \cup \{1\}. 
\end{equation*}
The $n$-th interval in this union is denoted by $I^n=[1-2^{-n},1
-2^{-n-1}]$. Let $s_n\colon I^n \to I$, $s_n(t)= 2^{n+1}(t- (1
-2^{-n}))$, for $n\in \N_0$. We define  
$H\: S^2\times I\ra \widehat S^2$ by 

\begin{align*}
  H(x,t) \coloneqq  H^{n+1}(x,s_n(t))
\end{align*}
 if  $x\in S^2$ and $t\in I^n$  for some  $n\in \N_0,$
and $H(x,t)=h(x)$ for $x\in S^2$ and $t=1$. 
We claim that $H$ is indeed an isotopy between $h_1$ and $h$ rel.~$f^{-1}(\post(f))$. 

Note that $H$ is well-defined, $H_1=h$, and  $H_{1-1/2^n}=h_{n+1}$ for $n\in \N_0$.  Moreover, $H_t$ is a homeomorphism for each $t\in I$, and $H_t|f^{-1}(\post(f))$ does not depend on $t$. 
To establish our claim,  it remains to verify  that $H$ is continuous. It is clear that 
$H$ is continuous at each point $(x,t)\in S^2\times [0,1)$. 

Moreover, as follows from the uniform convergence $h_n\to h$ as $n\to \infty$ and inequality \eqref{diamtracks}, we have $H_t\to H_1$ uniformly on $S^2$ as $t\to 1$.   This together with the continuity of $h=H_1$ implies the continuity of $H$ at points $(x,t)\in S^2\times I$ with $t=1$. 
\end{proof}

\begin{rem}
  \label{rem:betterandbetter} 
  The previous proof gives a procedure for approximating the
  conjugating map $h=h_\infty$. Indeed, we know that $H^n_t$ is an isotopy rel.\ $f^{-n}(\post(f))$ and so 
  the map $H^n_t$ is constant in $t$ on $f^{-n}(\post(f))$ for
each    $n\in \N_0$. This implies that
  $h_n=h_{n+1}=\dots =h_\infty$ on the set $f^{-n}(\post(f))$,
  and so the map $h_n$ sends the points in $f^{-n}(\post(f))$ to
  the ``right'' points in $g^{-n}(\post(g))$.  The isotopy
  $H_t^n$ then deforms $h_n$ to a map $h_{n+1}$ such that the
  points in $f^{-(n+1)}(\post(f))$ have the correct images in
  $g^{-(n+1)}(\post(g))$ as well, etc.  Since by expansion the
  union of the sets
  $$ \post(f)\sub f^{-1}(\post(f))\sub  f^{-2}(\post(f))\sub
  \dots$$ 
  is dense in $S^2$, this gives better and better approximations
  of the limit map $h_\infty$.
\end{rem}

The following fact, already mentioned in 
Section~\ref{sec:thurston-equivalence}, is an immediate consequence of the considerations in the proof of 
Theorem~\ref{thm:exppromequiv}.

\begin{cor}
  \label{cor:fg_eq_fngn_eq}
  \index{Thurston map!iterate of}
  \index{iterate of Thurston map}
  \index{F f@$F=f^n$}
  \index{Thurston!equivalent}
  Let $f\colon S^2\to S^2$ and $g\colon \widehat{S}^2 \to
  \widehat{S}^2$ be  Thurston maps. If $f$ and $g$ are (Thurston) equivalent, then 
  $f^n$ and $g^n$ are  equivalent for each $n\in \N$. 
\end{cor}

In general, it is not true that $f$ and $g$ are equivalent if  $f^n$ and $g^n$ are equivalent for some $n\ge 2$. 

\begin{proof} 
  We use the same notation as in the proof of
  Theorem~\ref{thm:exppromequiv}.  For the construction of the
  infinite tower of isotopies the assumption that $f$ and $g$ are
  expanding was not needed; so we also obtain such a tower under
  our given assumption that $f$ and $g$ are equivalent Thurston
  maps.

Let $n\in \N$ be arbitrary.
Then  $h_n=H^{n-1}_1$ is a homeomorphism such that 
$$  g^n\circ h_n =g^{n-1}\circ h_{n-1}\circ f= \dots= h_0\circ f^n. $$
Moreover, the homeomorphisms $h_0$ and $h_n$ are isotopic 
rel.\ $\post(f)$, because a suitable isotopy can be obtained by concatenating   the isotopies $H^0, \dots, H^{n-1}$. 
Hence $f^n$ and $g^n$ are equivalent as desired. \end{proof}

\section{Isotopies of Jordan curves}
\label{sec:isotopy-rel.-postf}\index{isotopy!of Jordan curve}

Let $X$ be a topological space, and $A,B,C\sub X$. We say that
$B$ is 
{\em isotopic to $C$ rel.\ $A$},\index{isotopy!rel.\ $A$} 
or {\em $B$ can be isotoped (or  deformed)  into $C$ rel.\ $A$}, if there
exists an isotopy $H\: X\times I\ra X$ rel.\ $A$ with $H_0=\id_X$
and $H_1(B)=C$ (see Section~\ref{sec:thurston-equivalence}).  This notion depends on the ambient
space $X$ containing the sets $A$, $B$, $C$.

In the following,  the ambient space for all isotopies will be a fixed  $2$-sphere $S^2$ 
equipped with a base metric. We will  study the problem 
 when two Jordan curves $J$ and
$K$ on $S^2$   passing through a given finite set $P$ of points in the same order can be deformed into each other by an isotopy of $\Sp$ rel.\ $P$. 
If $\#P\le 3$ this is always the case (see
Lemma~\ref{lem:deform<4} below). 

For $\#P\ge 4$ this is not always true as the example in Figure
\ref{fig:isotopcounter} shows. Here $K=S^1$ is the unit circle
and $P=\{1,\iu, -1,-\iu\}\subset S^1$. The Jordan curve $J$
(which contains $P$) is drawn with a thick line. The curves
$K=S^1$ and $J$ are not isotopic rel.\ $P$. In fact, $J$ can be
obtained from $S^1$ by a ``Dehn twist'' about a Jordan curve that separates the points $-\iu$ and $1$ from $\iu$ and
$-1$. Note that in this example we can make the Hausdorff
distance (see \eqref{eq:def_Hausdorffd}) between $J$ and $S^1$
arbitrarily small.

\ifthenelse{\boolean{nofigures}}{}{ 
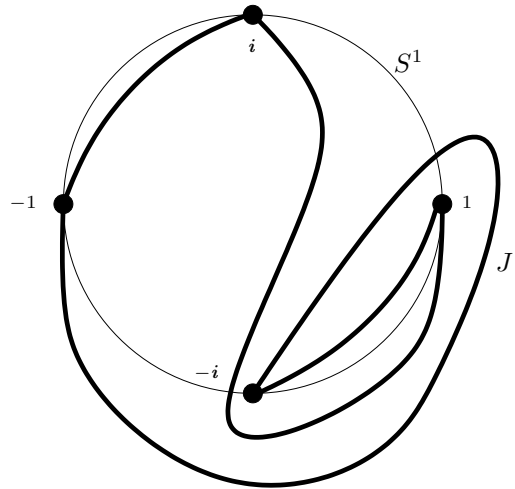
\begin{figure}
 \centering
 \begin{overpic}
   [width=6cm, 
   tics=20]{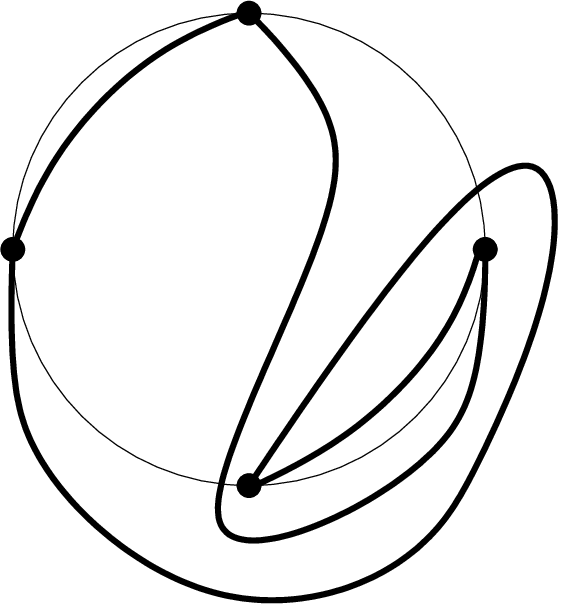}
   \put(70,86){$S^1$}
   \put(91,45){$J$}
   \put(40,90){$\scriptstyle{\iu}$}
   \put(-9,58){$\scriptstyle{-1}$}
   \put(29,23){$\scriptstyle{-\iu}$}
   \put(84,58){$\scriptstyle{1}$}
 \end{overpic}
 \caption{$J$ is not isotopic to $S^1$ rel.\ $\{1,\iu, -1, -\iu\}$.}
   \label{fig:isotopcounter}
\end{figure}
}

We will  need  the following statement. 

\begin{prop} \label{prop:isotreln}
 Suppose $J$ is a  Jordan curve in $\Sp$ 
 and $P\sub J$  a set consisting of  $n\ge 3$  distinct points $p_1, \dots, p_n, 
p_{n+1}=p_1$ in cyclic order on $J$. For $i=1,\dots, n$ let 
$\alpha_i$ be the unique  arc on  $J$ with endpoints $p_i$ and $p_{i+1}$
such that $\inte(\alpha_i)\sub J\setminus P$. Then there exists $\delta>0$ with the following property:

  Let  $K$ be  another Jordan curve in $\Sp$ 
passing through the  points $p_1, \dots, p_n$ in cyclic order, and let $\beta_i$ for $i=1, \dots, n$ be the  arc  with endpoints $p_i$ and $p_{i+1}$ such that $\inte(\beta_i)\sub J\setminus P$. 
If 
  \begin{equation*}
    \beta_i \subset \mathcal{N}_\delta(\alpha_i) 
  \end{equation*}
  for all $i=1, \dots, n$, then there exists an isotopy $H_t$ on 
  $\Sp$ rel.\ $P$ such that $H_0=\id_{\Sp}$ and 
  $H_1(J)=K$. 
\end{prop}

In other words, if the arcs $\beta_i$ of the Jordan curve $K$ are
contained in sufficiently small neighborhoods of the
corresponding arcs $\alpha_i$ of $J$, then one can deform $J$
into $K$ by an isotopy of $\Sp$ that keeps the points in $P$
fixed.  Even though this statement seems ``obvious'', a complete
proof is surprisingly difficult and involved. We will derive it from two
lemmas in \cite{Bu}.

\begin{lemma} \label{twoarcs}
Let $\Om\sub \Sp$ be a simply connected  region, $p,q\in \Om$ distinct points,
and $\alpha$ and  $\beta$ arcs in $\Om$ with endpoints $p$ and $q$. Then $\alpha$ is isotopic to  $\beta$ rel.\ $\{p,q\} \cup
\Sp\setminus\Om$. 
\end{lemma}

So arcs in a simply connected  region with the same endpoints can be deformed 
into each other so that the endpoints and the complement of the  region stay fixed. The lemma follows from   \cite[ A.6 Theorem~(ii), p.~413]{Bu}.

\begin{lemma}  \label{isotopylem}
 Suppose we have two Jordan curves $J$ and $K$ 
as in Proposition~\ref{prop:isotreln} such that for each $i=1, \dots, n$
the arc $\alpha_i$ is isotopic to  $\beta_i$ rel.\ $P$. 
Then $J$ is isotopic to $K$  rel.\ $P$.   
\end{lemma}

This is essentially \cite[A.5 Theorem, p.~411]{Bu}.

\begin{proof}[Proof of Proposition~\ref{prop:isotreln}]
For each arc $\alpha_i$ there exists a simply connected  region $\Om_i$ that contains 
$\alpha_i$ but does not contain any  element of  $P$ different from the endpoints of $\alpha_i$. There exists $\delta>0$ such that $\mathcal{N}_\delta(\alpha_i)\sub \Om_i$ for all
$i=1, \dots, n$. 
 Then by Lemma~\ref{twoarcs} every arc $\beta_i $ in $\mathcal{N}_\delta(\alpha_i)$ with the same endpoints as $\alpha_i$ can be isotoped  to
  $\alpha_i$ rel.\ $P$. 
 The proposition  now follows from Lemma~\ref{isotopylem}.
 \end{proof}
 
If $\#P\le 3$ in Proposition~\ref{prop:isotreln}, then $J$ can always be  isotoped to $K$ rel.\ $P$.
 
\begin{lemma} 
  \label{lem:deform<4} 
  Suppose $J$ and $K$ are Jordan curves in $S^2$ and
  $P\sub J\cap K$ is a set with $\#P\le 3$. Then $J$ is isotopic
  to $K$ rel.\ $P$.  
\end{lemma}
 
\begin{proof} 
  Suppose first that $P$ consists of  three distinct
  points $p_1$, $p_2$, $p_3$. Define the arcs $\alpha_i$ and
  $\beta_i$ as in Proposition~\ref{prop:isotreln}. Then for each
  $i=1, 2,3$ the arcs $\alpha_i$ and $\beta_i$ have the same
  endpoints 
$p_i$ and $p_{i+1}$, and are contained 
in the simply
  connected region $\Om_i=S^2\setminus\{p_{i+2}\}$, where indices
  are understood modulo $3$.  Hence by Lemma~\ref{twoarcs} each
  arc $\alpha_i$ is isotopic to $\beta_i$ rel.\ $P$.  Again
  Lemma~\ref{isotopylem} implies that $J$ is isotopic to $K$
  rel.\ $P$.

If  $\#P\le 2$, we may assume that $S^2=\CDach$. Then by applying the first part of the proof (by adding auxiliary points to $P$) one   sees that both $J$ and $K$ are isotopic to circles in $\CDach$ rel.\ $P$. Hence $J$ is isotopic to  $K$ rel.\ $P$.    \end{proof}

\begin{lemma} 
  \label{lem:homeo} 
  Let $S^2$ and $\widehat S^2$ be oriented $2$-spheres, and
  $P\sub S^2$ be a set with $\#P\le 3$. If
  $h_0\: S^2\ra\widehat S^2$ and $h_1\: S^2\ra\widehat S^2$
  are orientation-preserving homeomorphisms with
  $h_0|P=h_1|P$, then $h_0$ and $h_1$ are isotopic
  rel.\ $P$.
\end{lemma}

\begin{proof} 
  The statement is essentially well known. For the sake of
  completeness we will give a proof, but will leave some of the
  details to the reader. These details can easily be filled in
  along the lines of the proof of
  Lemma~\ref{lem:isocellhomeo}~\ref{item:isocellhomeo3}. 
   
  By considering $h\coloneqq h_1^{-1}  \circ  h_0$ one can reduce
  the lemma to the case where $S^2=\widehat S^2$ and
  $h_1=\id_{S^2}$. Then $h$ is an orientation-preserving 
  homeomorphism on $S^2$ fixing  the points in $P$, and we
  have to show that $h$ is isotopic to $\id_{S^2}$ rel.\
  $P$.  We first assume that $\#P=3$.
   
  Pick a Jordan curve $K\sub S^2$ with $P\sub K$, and let
  $J=h(K)$. Then $P\sub J\cap K$, and so by
  Lemma~\ref{lem:deform<4} the Jordan curve $J$ can be isotoped
  into $K$ rel.\ $P$. This implies that $h$ is isotopic
  rel.\ $P$ to an orientation-preserving  homeomorphism $\varphi_1$ on $S^2$ with
  $\varphi_1(K)=K$ and $\varphi_1|P=\id_P$. Since $\varphi_1$ is orientation-preserving and fixes the points in $P$, it preserves the orientation of $K$ represented by some cyclic order of the points in $P$. This implies that $\varphi_1$ sends each of the two Jordan regions bounded by $K$ to itself.

   Let $e$ be one of the three subarcs of $K$
  determined by $P$.  Since $\varphi_1$ fixes the three points in
  $P$, this map restricts to a homeomorphism of $e$ that does not
  move the endpoints of $e$. Hence on $e$ the map $\varphi_1$ is
  isotopic to the identity on $e$ rel.\ $\partial e$.  
  
  By pasting the isotopies on these arcs together, 
we can construct
  an isotopy $H\: K\times I\ra K$ rel.\ $P$ such that
  $H_0=\id_{K}$ and $H_1=\varphi_1|K$.  One can extend $H$ to
  each of the two Jordan regions bounded by $K$ to obtain an
  isotopy $\overline{H}\: S^2\times I\ra S^2$ rel.\ $P$ such that
  $\overline{H}_0=\id_{S^2} $ and $\overline{H}(p,t)=H(p,t)$ for
  all $p\in K$ and $t\in I$.  Then
  $\varphi_2\coloneqq \overline{H}_1$ is a homeomorphism on $S^2$
  that is isotopic to $\id_{S^2}$ rel.\ $P$ such that
  $\varphi_1|K=\varphi_2|K$.  This implies that $\varphi_1$ and
  $\varphi_2$ are isotopic rel.\ $K\supset P$ (here it is
  important that $\varphi_1$ and $\varphi_2$ do not interchange
  the two Jordan regions bounded by $K$).  If $\sim$ indicates
  that two homeomorphisms on $S^2$ are isotopic rel.\ $P$, then
  we have $h\sim \varphi_1 \sim \varphi_2\sim \id_{S^2}$, and so
  $h\sim \id_{S^2}$ as desired.

  %
   
  If $\#P \le 2$, then we pick a set $P'\sub S^2$ with $\#P'=3$
  and $P'\supset P$.  By the first part of the proof it suffices
  to find an isotopy rel.\ $P$ of the given map $h$ to a
  homeomorphism $h'$ that fixes the points in $P'$. It is
  clear that such an isotopy can always be found; for an explicit
  construction one can assume that $S^2=\CDach$ and obtain the
  desired isotopy by postcomposing $h$ with  a suitable
  continuous family of M\"obius transformations, for example.
\end{proof}

The following lemma will be  crucial for the proof of the
uniqueness statement 
for invariant Jordan curves.  
In its proof we will use the following topological fact: if $D$ is a $2$-dimensional cell and  $\varphi\: D
\ra S^2$ is a continuous map such that $\varphi| \partial D$ is injective, then the set 
$\varphi(\inte(D))$
contains one of the two complementary components of the Jordan curve $\varphi(\partial D)$. 
Indeed, by applying the Sch\"onflies theorem and using auxiliary homeomorphisms
we can  reduce to the case where $D= \overline \D$, $S^2=\CDach$,  $\varphi|\partial \D=\id_{\partial \D}$, and $\infty\notin \varphi(D)$. 
Then $\D\sub \varphi(\D)$.  This    follows from a simple degree argument and the statement can be generalized to higher dimensions; for an elementary exposition of this and related facts 
in dimension $2$ see \cite{Bur}, in particular \cite[Corollary~3.5]{Bur}.   

\begin{lemma} \label{lem:isoJcin1ske}  Let $\DD$ be a cell decomposition of $S^2$ with $1$-skeleton $E$ and vertex set ${\bf V}$, and suppose  that every tile in $\DD$ contains at least three vertices on its boundary. If  $J$ and  $K$ are Jordan curves that are both contained in $E$ and are isotopic rel.\  ${\bf V}$,   then $J=K$.  \end{lemma}

\begin{proof} Let $H\: S^2\times I\ra S^2$ be an isotopy rel.\ ${\bf V}$ 
such that $H_0=\id_{S^2}$ and $H_1(J)=K$. 

Note that if $M\sub S^2$ is a set disjoint from ${\bf V}$, then it remains disjoint from ${\bf V}$ during the isotopy, i.e., if $M\cap {\bf V}=\emptyset$, then $H_t(M)\cap {\bf V}=\emptyset$ for all 
$t\in I$. This follows from the fact that each map $H_t$, $t\in [0,1]$,  is a homeomorphism on $S^2$ with $H_t|{\bf V}=\id_{{\bf V}}$.

Let $e$ be an edge in $\DD$. We claim that if $H_1(e)\sub E$, then $H_1(e)=e$. 
First note that $H_1(e)$ is an edge in $\DD$. 
Indeed, since $\partial e\sub {\bf V}$ and the isotopy $H$ does not move vertices,
 the arc $H_1(e)$ has the same endpoints as $e$.  Moreover,  $\inte(e)\cap {\bf V}=\emptyset$, and so $H_1(\inte(e))\cap {\bf V}=\emptyset$  by what we have just seen. So $H_1(\inte(e))$ is a connected set in the $1$-skeleton $E$ of $\DD$ disjoint from the $0$-skeleton ${\bf V}$.  By  Lemma~\ref{lem:opencells}  there exists an edge 
$e'$ in $\DD$  with $H_1(\inte(e))\sub \inte(e')$.  Since the endpoints of $H_1(e)$ lie in 
${\bf V}$, this implies that $e'=H_1(e)$. 

To show that $e'=e$ we argue by contradiction and assume that $e\ne e'$. Then  $e$ and $e'$ have the same  endpoints, but no other points in common. Therefore $\alpha=e\cup e'$ is a Jordan curve that contains two vertices, namely the endpoints of $e$ and $e'$, but no other vertices.  Let $\Om_1$ and $\Om_2$ be the two open Jordan regions that form the complementary components of $\alpha$.  Then both regions $\Om_1$ and $\Om_2$ contain vertices. 

To see this, note that the interior of every tile $X$ is a connected set disjoint from the $1$-skeleton $E$, and hence also disjoint from $\alpha$. Therefore $\inte(X)$ is contained in $\Om_1$ 
or $\Om_2$. Moreover, since the union of the interiors of tiles is dense in $S^2$, both regions $\Om_1$ and $\Om_2$ must contain the interior of at least one tile.

Now consider $\Om_1$, for example, and pick a tile $X$ with $\inte(X)\sub \Om_1$.
Then by our hypotheses the set $X\sub \overline{\Om}_1=\Om_1\cup\alpha$ contains at least three vertices. Since only two of them can lie on $\alpha$, the set $\Om_1$ must contain a vertex.  Similarly, $\Om_2$ must contain at last one vertex. 

A contradiction can now be obtained from  the fact that during the isotopy $H$ the set $\inte(e)$
remains  disjoint from the set of vertices, but on the other hand it has to sweep out one of the regions $\Om_1$ or $\Om_2$ and hence it meets a vertex. 

To make this rigorous, we apply the topological fact mentioned
before the statement of the lemma. Let $u$ and $v$ be the
endpoints of $e$. We collapse $\{u\}\times I$ and $\{v\}\times I$
in $e\times I$ to obtain a set $D$. 
Formally $D$
is the quotient of $e\times I $ obtained by identifying all
points $(u,t)$, $t\in I$, and by identifying all points $(v,t)$,
$t\in I$. Then
$D$ is a $2$-dimensional cell.  Since the isotopy $H$ does not
move the points $u$ and $v$, the map $(p, t)\mapsto H_t(p)$ on
$e\times I$ induces a continuous map $\varphi\: D \ra S^2$.
Moreover, $\varphi | \partial D$ is a homeomorphism of
$\partial D$ onto $\alpha$.  Hence $\Om_1$ or $\Om_2$ is
contained in the set
$$\varphi(\inte(D))=\bigcup_{t\in (0,1)} H_t(\inte(e)). $$
In particular, the set $\varphi(\inte(D))$ contains a vertex. This is a contradiction, because we know that 
no set $H_t(\inte(e))$, $t\in I$,  meets ${\bf V}$.  Thus $H_1(e)=e$ as desired.

Having verified the statement about edges, we can now easily   
   show  that $J=K$. 
    Indeed,  $J$ is a union of edges in $\DD$; to see this, consider the components of the set
    $J\setminus {\bf V}$. If  $\ga$ is such a component, then $\overline \ga \setminus \ga\sub 
    {\bf V}$. Moreover,  $\ga$ is contained in the $1$-skeleton $E$, and does not meet the $0$-skeleton ${\bf V}$.  Again by   Lemma~\ref{lem:opencells}  the set $\ga$ must  be contained in the interior $\inte(e)$ 
    of some edge $e$. This is only possible if $\ga=\inte(e)$. Hence $\overline \ga=e$. Since $J$ is the union of the closures of these components $\ga$, it follows that 
   $J$ is the union of edges $e$. For each such edge $e$ we have $H_1(e)\sub K\sub E$ and so $H_1(e)=e$ by the first part of the proof. This implies $J\sub K$. Since $J$ and $K$ are Jordan curves,  the desired identity $J=K$ follows. 
    \end{proof}

\section{Isotopies and cell decompositions}
\label{sec:graphs}

The main result in this section is Lemma~\ref{lem:isorelP} which gives a criterion when a 
Jordan curve $\CC$ in a $2$-sphere $S^2$ can   be isotoped relative to a finite  set $P\sub \CC$ into the  $1$-skeleton of a given cell decomposition $\DD$ of $S^2$.   We first  discuss some  facts about graphs that are 
needed in the proof.
 Since all the graphs we consider will be embedded  in a $2$-sphere, we base the concept of a graph on a topological definition rather than  a combinatorial one as usual (in  Chapter~\ref{cha:latt-maps-comb}
it will be more convenient to adopt the combinatorial viewpoint).

A  {\em (finite) graph}\index{graph}  is a compact Hausdorff space  $G$
equipped with a fixed cell decomposition $\DD$ such that $\dim(c)\le
1$ for all $c\in \DD$. The cells $c$ in $\DD$ of dimension $1$ are
called the {\em edges} of the graph, and the points $v\in G$ such that
$\{v\}$ is a $0$-dimensional cell in $\DD$ the {\em vertices} of the
graph.  Note that we do allow multiple edges, i.e., two or more edges
with the same endpoints $v,w$. Loops however, meaning edges where the
two endpoints agree, are not allowed according to our definition.

 An {\em oriented edge} $e$ in a graph is an edge, where one of the vertices  in $\partial e$ has been chosen as the {\em initial point} and the other 
 vertex as the {\em terminal point} of $e$. An {\em edge
   path}\index{edge!path} in $G$
 is a finite sequence $\alpha$   of oriented edges $e_1, \dots, e_N$
 such that  
 the terminal point of $e_i$ is the initial point of $e_{i+1}$ for $i=1, \dots, N-1$.
 We denote by $|\alpha|=e_1\cup \dots \cup e_N$ the underlying set of the edge path.  
 The edge path  $\alpha$ {\em joins} the vertices $a,b\in G$  if the initial point of $ e_1$ is $a$ and the terminal point of $e_N$ is $b$. The number $N$ is called the {\em length} of the edge path. The edge path is called {\em simple} if 
 $e_i$ and $e_j$ are disjoint for 
$1\le i<j\le N$ and $j-i\ge2$, and  $e_i\cap e_j$ consists  of  precisely one point (the terminal point of $e_i$ and initial point of $e_j$) when $j=i+1$. 
If the edge path $\alpha$ is simple, then $|\alpha|$ 
is an arc.  The edge path is called a {\em loop} if the terminal point of $e_N$ is the initial point of $e_1$. 

A graph is connected (as a topological space) if and only if any two vertices $a,b\in G$, $a\ne b$, can be joined by an edge path. 
 The {\em combinatorial distance} of two vertices $a$ and $b$ in a connected graph $G$ is defined as the minimal length of all edge paths joining  the points (interpreted as $0$ if  $a=b$).  The vertices $a,b\in G$ are called {\em neighbors} if their combinatorial distance is equal to $1$, i.e., if there exists an edge $e$ in $G$ whose endpoints are $a$ and $b$.  A vertex  $q\in G$ is called a {\em cut point} 
of  $G$ if $G\setminus\{q\}$ is not connected.  A vertex  $q\in G$ is not a cut point  if and only if all vertices $a,b\in G\setminus \{q\}$, $a\ne b$, can be joined by an edge path $\alpha$ with $q\notin |\alpha|$.

\begin{lemma}
  \label{lem:simple}
  Let $G$ be a connected graph without cut points. Then for all vertices 
  $a,b,p\in G$ with $a\ne b$ there exists a simple edge path $\ga$ in $G$ with $p\in |\ga|$ that joins $a$ and $b$.  \end{lemma}

\begin{proof}
  Since $G$ is connected, there exist edge paths in $G$ joining  $a$ and $b$. By removing loops from such a path if necessary, we can also obtain such an  edge path in $G$ that is simple.
  Among all such simple 
  paths, there is one  that contains a vertex with minimal  combinatorial distance  to $p$.
 More precisely,  there exists a simple edge  path $\alpha$ in $G$ with endpoints $a$ and $b$, and a vertex  $q\in |\alpha|$ such that the combinatorial distance $k\in \N_0$ of $q$ and $p$ is minimal among all combinatorial distances between $p$ and vertices  on simple paths joining   $a$ and $b$. If $k=0$ then  $q=p$ and we can take $\gamma=\alpha$.
  
 We will show that the alternative case $k\ge 1$ leads to a contradiction. 
 By definition of  combinatorial distance, there exists an edge path joining 
 $q$ to $p$ consisting of $k\ge 1$ edges. The second  vertex 
 $q'$ on this path as traveling from $q$ to $p$ is a neighbor of $q$ 
 whose combinatorial distance to $p$ is $k-1$ and hence strictly smaller than the combinatorial distance of $q$ to $p$. In particular, $q'\notin  
 |\alpha|$ 
 by choice of $q$ and $\alpha$. 
 We will obtain the desired contradiction if we can show that there exists a simple edge  path $\sigma$ in $G$ that joins  $a$ and $b$ and passes through $q'$.

\ifthenelse{\boolean{nofigures}}{}{ 
 \begin{figure}
  \centering
  \begin{overpic}
    [width=10cm, 
    tics=20]{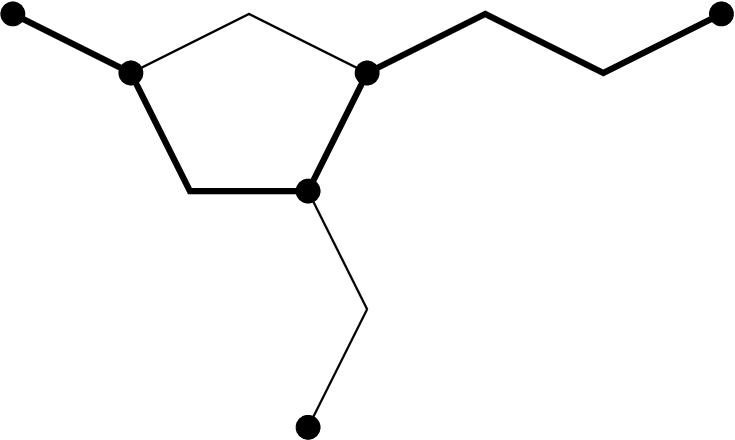}
    \put(0,53){$a$}
    \put(15,45.5){$r$}
    \put(45,1){$p$}
    \put(52,47){$q$}
    \put(44.5,33){$q'$}
    \put(97,52){$b$}
    \put(30,58){$\alpha$}
    \put(71,50){$\sigma$}
    \put(20,35){$\beta$}
  \end{overpic}
  \caption{Constructing a path through $a,b,p$.}
  \label{fig:abp_path}
\end{figure}
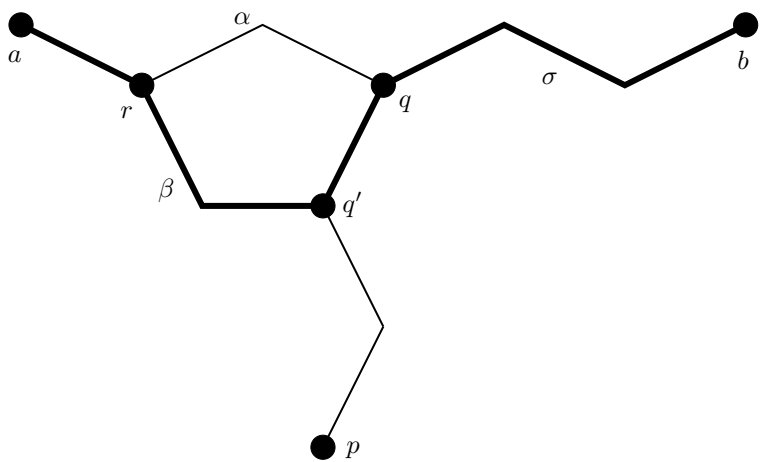
}

 For the construction of $\sigma$  
 we apply our assumption that $G$ has no cut points; in particular, $q$ is no cut point and hence there exists 
 an edge  path $\beta$ with $q\notin |\beta|$ that joins  
 $q'$ to a vertex in the (non-empty) set $A=|\alpha|\setminus \{q\}$. We may assume 
 that $\beta$ is simple and that the  endpoint $r\ne q'$ of $\beta$ is the only point in  $|\beta|\cap A$. 
 
Moreover, we may assume that $r$ lies between $a$ and $q$ on the path $\alpha$ 
(the argument in  the other case where  $r$ lies between $q$ and $b$ is similar). Now let $\sigma$ be the edge  path obtained by traveling from $a$ to $r$ along $\alpha$, then from $r$ to $q'$ along $\beta$, then from $q'$ to $q$ along an edge
(this is possible since $q$ and $q'$ are neighbors), and finally from $q$ 
to $b$ along $\alpha$. See the  illustration 
in Figure~\ref{fig:abp_path}.  Then $\sigma$ is a simple edge path in $G$ that 
passes through $q'$ and has the endpoints $a$ and $b$.      This gives the desired contradiction. 
    \end{proof}

Now let $S^2$ be a $2$-sphere, and $\DD$ be a cell decomposition of $S^2$. We denote the set of tiles, edges, and vertices in $\DD$ by $\X$, $\E$, and ${\bf V}$, respectively.  In the following  the terms cell, tile, etc., refer to elements of these
sets. 

Let $M\sub \X$ be a set of tiles. We denote by $|M|$ its underlying set; so
$$ |M|=\bigcup_{X\in M} X. $$
  The  set  
  \begin{equation}
    \label{eq:def_GM}
    G_M\coloneqq \bigcup_{X\in M} \partial X   
  \end{equation}
  admits a natural cell
  decomposition consisting of all cells contained in $G_M$.
  Obviously, no such cell can be a tile, so with this cell
  decomposition $G_M$ is  a graph. 

  Recall from Definition~\ref{def:e-chain} that a sequence $X=X_1, \dots, X_N=Y$ of
  tiles  is an 
{\em $e$-chain}\index{e-chain@$e$-chain}\index{chain!$e$-} 
if $X_i\ne X_{i+1}$ 
and  there exists
  an edge $e_i$ with
  $e_i\sub \partial X_{i}\cap \partial X_{i+1}$ for $i=1, \dots, N-1$. It \emph{joins}
  the tiles $X$ and $Y$. A set $M$ of tiles is {\em
    $e$-connected}\index{e-connected@$e$-connected} if every two
  tiles in $M$ can be joined by an $e$-chain consisting of tiles
  in $M$.

\begin{lemma}
  \label{lem:nocut}
  Let $M\sub \X$ be a set of tiles  that   
  is  $e$-connected. Then the graph 
 $G_M$ is connected and has no cut points.  
 \end{lemma}

\begin{proof}
Let   
  $a,b\in G_M$ be arbitrary vertices with $a\ne b$. We can  pick
  tiles $X$ and $Y$ in $M$ such that  $a$ is a vertex in  $X$ and
  $b$ is a vertex in  $Y$.  
By assumption there exists an $e$-chain $X_1, \dots, X_N$ 
in $M$ with 
 $X_1=X$ and $X_N=Y$. The vertices of a tile $X_i$ lie in 
 $G_M$; they subdivide the Jordan curve $\partial X_i$ such that successive vertices on $\partial X_i$ are connected by an edge
 and are  hence neighbors in $G_M$. An edge  path 
 $\alpha$ in $G_M$ joining  $a$ and $b$ can now be obtained as
 follows: starting from 
$a\in \partial X_1$,  
use edges  on the 
boundary of $X_1$ 
to find an edge  path in $G_M$ that joins 
$p_1=a$ to a vertex  $p_2$ of $X_2$. 
This is possible, since 
$X_1$ and $X_2$ have a common edge 
and
hence at least two common vertices. Then run from 
$p_2$ along edges on $\partial X_2$ to a vertex $p_3$ of $X_3$, 
and so on.  Once we arrived at a vertex $p_N$ of $X_N$, we can
reach $b$ by running from $p_N$ to $p_{N+1}\coloneqq b$ along 
edges on $\partial X_N$. 
 In this way  we obtain an edge  path $\alpha$ in $G_M$ that joins  $a$ and $b$. 
 
 A slight refinement of this argument also shows that we can construct the 
 path $\alpha$ so that it avoids any given vertex  $q$ in $G_M$ distinct 
 from $a$ and $b$. Indeed, choose 
$p_1=a$ as before. Since $X_1$ and $X_2$ 
have at least two vertices in common, we can pick a 
common vertex $p_2$ of $X_1$ and $X_2$ 
that is distinct from $q$. 
 There exists an 
arc on $\partial X_1$ 
(possibly degenerate) that does not  contain  $q$ and 
joins $p_1$ and $p_2$. 
This arc (if non-degenerate) consists of edges  and if  we follow
these  edges, we obtain an edge  path in $G_M$ that does not
contain $q$ and 
joins  $p_1$ and $p_2$. 
In the same way we can find an edge path 
 in $G_M$ that avoids $q$ and joins  
$p_2$ to a vertex $p_3\in \partial X_2\cap \partial X_3$, 
and so on. Concatenating  all these edge paths we get a path $\alpha$ as desired. 
 
 This shows that $G_M$ is connected and has no cut points. 
 \end{proof}

\begin{lemma}
  \label{lem:paththrough}
  Let $M\sub \X$ be a set of  tiles that is $e$-connected, 
  and let $a,b,p\in |M|$ be distinct vertices.  Then there 
  exists a simple edge path  $\alpha$ in $G_M$ with $p\in |\alpha|$
 that  joins $a$ and $b$.  \end{lemma}

In particular, this applies  if $M$ consists of a single $e$-chain.

\begin{proof} This follows from   Lemma~\ref{lem:nocut} and 
Lemma~\ref{lem:simple}.  \end{proof}

 \begin{lemma}
  \label{lem:echain}  Let $\ga \:J\ra S^2$ be  a path in $S^2$ defined on an   interval $J\sub \R$ and 
  $M=M(\ga)$  be the set of tiles having non-empty intersection with
  $\ga$. Then $M$ is $e$-connected.  
  \end{lemma} 
 
 \begin{proof}  We first prove the following claim.
 If $[a,b]\sub \R$,  
  $\alpha\:[a,b]\ra S^2$ is  a path, and $X$ and $Y$ are  tiles with $\alpha(a)\in X$ and $\alpha(b)\in Y$,  then there exists  an $e$-chain $X_1=X, X_2,\dots, X_N=Y$  such that 
  $X_i\cap \alpha \ne \emptyset $ for all $i=1, \dots, N$.

 In the proof of this claim,  we call an $e$-chain $X_1, \dots, X_N$ 
{\em admissible} if $X_1=X$ and  $X_i\cap \alpha \ne \emptyset $ for all $i=1, \dots, N$. So we want to find an admissible $e$-chain whose last tile is $Y$. 

 Let $T\sub [a,b]$ be the set of all points $t\in [a,b]$ for which there exists an 
 admissible $e$-chain $X_1, \dots, X_N$ with $\alpha(t)\in X_N$. 
We first want  to show that $b\in T$.

Note that the set $T$ is closed. Indeed, suppose that $\{t_k\}$ is a sequence in $T$ 
with $t_k\to t_\infty \in [a,b]$ as $k\to \infty$. Then for each $k\in \N$ there 
exists an admissible $e$-chain $X^k_1, \dots , X^k_{N_k}$ with 
$\alpha(t_k)\in X^k_{N_k}$. Define $Z_k=X^k_{N_k}$ to be the last 
tile in this chain. Since there are only finitely many tiles, there exists 
one tile, say $Z$, among the tiles $Z_1$, $Z_2$, $Z_3$, $\dots$
that appears infinitely often in this sequence. 
Then we have $\alpha(t_k) \in Z$ for infinitely many $k$. Since tiles are closed,  we conclude that $\alpha(t_\infty) =\lim_{k\to\infty} \alpha(t_k)\in Z$. 
By definition of $Z$ there exists an admissible $e$-chain $X_1, \dots, X_N$ with $X_N=Z$. Then $\alpha(t_\infty)\in Z=X_N$, and so $t_\infty\in T$. 

Obviously,  $a\in T$ and so $T$ is non-empty. Since  $T$ is also closed, the set $T$ has a maximum, say
$m\in [a,b]$.   We have to show   that $m=b$; 
we will see  that the assumption $m<b$ leads to a contradiction. 

We consider $p\coloneqq \alpha(m)$. Then  there exists an admissible $e$-chain $X_1, \dots, X_N$ with $p\in Z\coloneqq X_N$. 

If $p\in \inte(Z)$, then $\alpha(t)\in Z$ and so $t\in T$ for $t\in(m,b]$  close to $m$. This is impossible by definition of $m$.

If $p$ does not belong to $\inte(Z)$, then  $p$ must be a boundary point
of $Z$. Suppose first that $p$   is in the interior of an edge 
$e\sub \partial Z$. By Lemma~\ref{lem:specprop}~\ref{item:prop_cell4} there exists precisely one tile $Z'$ distinct  from $Z$ such that  $e\sub \partial Z'$. 
Moreover,  $Z\cup Z'$ is a neighborhood of $p$, and so  
points  $\alpha(t)$ with  $t\in (m,b]$ close to $m$ belong to $Z$ or $Z'$.
Since $Z'$ contains $p$ and hence meets $\alpha$, and $Z$ and $Z'$ share an edge, $X_1, \dots, X_N=Z, Z'$ is an admissible $e$-chain. It follows that 
$t\in T$ for $t\in (m,b]$ close to $m$.  Again this is impossible by definition of 
$m$.

If $p$ is  a boundary point
of $Z$, but not  in the interior of an edge, then $p$ is a vertex. 
The tiles  in the cycle of $p$ form a  neighborhood of $p$,
and so a point $\alpha(t)$ for some $ t\in (m,b]$ close to $m$ will belong to a tile $Z'$ 
in  the cycle of $p$.   It follows from  Lemma~\ref{lem:specprop}~\ref{item:prop_cell5} that any two tiles in the cycle 
of a vertex can be joined  by an $e$-chain consisting of tiles in the cycle. 
Hence there exists an $e$-chain $Z=Z_1, \dots, Z_K=Z'$ such that 
$p\in Z_j$ for $j=1, \dots, K$. 
In particular, $\alpha\cap Z_j\ne \emptyset $ for $j=1, \dots, K$, and so  $X_1, \dots, X_N=Z=Z_1, \dots, Z_K=Y'$ is an admissible $e$-chain. 
Since $\alpha(t)\in Z'=Z_K$, we have $t\in T$, again a contradiction.

We have exhausted all possibilities proving that $b\in T$ as
desired. This implies that  there exists an admissible $e$-chain $X_1=X, \dots,
X_N$ with $\alpha(b)\in X_N$. If $X_N=Y$, then we are done. If
$X_N\ne Y$, then $\alpha(b)\in \partial X_N\cap \partial Y$, and
so $\alpha(b)$ is an interior point of an edge $e$ with
$e\sub \partial X_N\cap \partial Y$, or $\alpha(b)$ is a vertex.
As in the first part of the proof, one can then extend the admissible $e$-chain $X_1=X, \dots, X_N$  to obtain an admissible $e$-chain whose last tile is $Y$.  The claim made in the beginning of the proof  follows. 

This claim now easily implies the statement of the lemma. Indeed, let $X,Y\in M=M(\ga)$ be arbitrary. Then there exist $a,b \in J$ with 
  $\ga(a)\in X$ and $\ga (b)\in Y$. If $a\le b$, then we apply the  claim  to the path   $\alpha=\ga|[a,b]$, and if 
  $b\le a$ to the path $\alpha=\ga|[b,a]$. This shows that  we can  
  find an $e$-chain in $M$ that joins   $X$ and $Y$. 
  \end{proof}

 For the formulation of the next statement,  we need a slight extension   of Definition~\ref{def:connectop}.
Let $\CC\sub S^2$  be a Jordan curve,  and $P\subset \CC$ be a finite set with  
   $\#P\geq 3$.   The points in $P$ divide $\CC$ into subarcs 
  that have endpoints in $P$, but whose interiors 
   are disjoint from $P$.  
   We say  that a (not necessarily connected) set $K\sub S^2$ {\em joins opposite sides of $(\CC,P)$}\index{joining opposite sides}
   if $\#P\ge 4$ and $K$ meets two of these  arcs    that are non-adjacent (i.e., disjoint), or if $\#P=3$ and $K$ meets all of these 
    arcs (in this case there are three arcs).   
    
In the following lemma and its proof, metric notions refer to some fixed base metric on $S^2$.

\begin{lemma}
  \label{lem:isorelP}
  Let $\CC\sub S^2$ be a Jordan curve, and $P\subset \CC$ be  a
  finite set with $k\coloneqq \#P\geq 3$.  Then there exists
  $\epsilon_0 >0$ satisfying the following condition:
  
  Suppose that $\DD$ is a cell decomposition of $S^2$ with vertex
  set $\mathbf{V}$ and $1$-skeleton $E$. If
  $P\subset \mathbf{V}$ and
  \begin{equation*}
     \max_{c\in \DD} \diam (c)< \epsilon_0, 
  \end{equation*}
  then there exists a Jordan curve $\CC'\sub E$ that is isotopic
  to $\CC$ rel.\ $P$, and has the property that no tile in $\DD$
  joins opposite sides of $(\CC', P)$.
\end{lemma}

\begin{proof}  We fix an orientation of $\CC$ and let $p_1, \dots, p_k$ be the points in $P$ in cyclic order on $\CC$. The
 points in $P$ divide $\CC$ into subarcs $\CC_1, \dots, \CC_k$ such that for $i=1, \dots, k$ the arc $\CC_i$ has the endpoints $p_i$ and $p_{i+1}$ and has interior disjoint from $P$. Here and in the following 
the index $i$ is understood modulo $k$, i.e., $p_{k+1}=p_1$, etc. 
Note that  $\CC_i\cap \CC_{i+1}=\{p_{i+1}\}$ for $i=1, \dots, k$. 
There exists a number $\delta_0>0$ such that no set $K\sub S^2$ with $\diam(K) <\delta_0$ joins opposite sides of $(\CC,P)$ (this can be seen as in the discussion after \eqref{defdelta}).  

Now choose $\delta>0$ as in Proposition~\ref{prop:isotreln} for $J=\CC$ (and $n=k$). We may assume that $3\delta<\delta_0$. 
We  break up $\CC$ into
 subarcs 
\begin{equation}
\label{eq:alph_gamma}
  \alpha_1, \gamma_1, \alpha_2, \gamma_2, \dots, \alpha_{k}, \gamma_k,
  \alpha_1,
\end{equation}
arranged in cyclic order on $\CC$, such that $p_i$ is an interior
point of $\alpha_i$ and we have $\alpha_i\sub B(p_i, \delta/2)$
for each $i=1,\dots, k$. The arcs in (\ref{eq:alph_gamma}) have
disjoint interiors, and two arcs have an endpoint in common if
and only if they are adjacent in this cyclic order in which case
they share one endpoint. So each ``middle piece'' $\gamma_i$ does
not contain any point from $P$ and is contained in the interior
of $\CC_i$.

We choose $0<\eps_0<\delta/4$ so small that the distance
between non-adjacent arcs in (\ref{eq:alph_gamma}) 
is $\ge 10 \eps_0$ and so that 
$$ \dist(p_i, \gamma_{i-1}\cup \gamma_i)\ge 10 \eps_0$$
for $i=1, \dots, k$.

 Now suppose we  have a cell  decomposition $\DD$ of $S^2$ such that $P$ is contained in the vertex set $\mathbf{V}$ of $\DD$  and 
 $$\max_{c\in \DD} \diam (c)  < \epsilon_0.  
$$

Our
goal is to find a Jordan curve $\CC'\sub S^2$ consisting  of  arcs
$\CC_i'$  that are unions of edges, have endpoints $p_i$ and
$p_{i+1}$,  and satisfy 
$$\CC'_i\sub \mathcal{N}_\delta(\CC_i) $$  for
$i=1, \dots, k$.

Let $\mathbf{A}_i$ be the set of all tiles intersecting $\alpha_i$ and
$\mathbf{C}_i$ be the set of all tiles intersecting $\gamma_i$ for
 $i=1,\dots, k$. 
Recall that for a given  set of tiles $M$, we denote by $\abs{M}$ the
union of tiles in $M$. Let $A_i\coloneqq  \abs{\mathbf{A}_i}$ and $C_i
\coloneqq
\abs{\mathbf{C}_i}$. 

Note that
\begin{equation*}
  A_i \subset \mathcal{N}_{\epsilon_0}(\alpha_i) \quad 
 \text{and} 
  \quad 
  C_i \subset \mathcal{N}_{\epsilon_0}(\gamma_i). 
\end{equation*}

Moreover, 
\begin{align}
  \label{eq:ABCDinNd}
  A_i\cup C_i\cup A_{i+1}&\subset
  \mathcal{N}_{\epsilon_0}(\alpha_i)
  \cup \mathcal{N}_{\epsilon_0}(\gamma_i)
  \cup \mathcal{N}_{\epsilon_0}(\alpha_{i+1}) 
 \\ \notag
  &\subset 
  B(p_i, \delta) \cup \mathcal{N}_{\epsilon_0}(\gamma_i) \cup
  B(p_{i+1},\delta)  
  \\ \notag
  &\subset \mathcal{N}_\delta(\CC_i),   
\end{align}
and the natural cyclic order of these sets is 
\begin{equation}\label{setlist}
A_1, C_1, A_2, C_2,\dots,A_k,  C_k, A_1. 
\end{equation}
By choice of $\eps_0$ we know that if two of the sets in (\ref{setlist})
are not adjacent in the cyclic order, then their distance  is $\ge
8\eps_0$ and so their intersection is empty.  Moreover, for $i=1,
\dots, k$ the only one of these sets that contains $p_i$ is $A_i$. 

The
construction that now follows is illustrated in
Figure~\ref{fig:pf_des_Grauens}. Here the two large dots represent two
points $p_i,p_{i+1}$ and the thick line the curve $\CC$.


 
\ifthenelse{\boolean{nofigures}}{}{
  \begin{figure}
    \centering
    \vspace{0.3cm}
    \begin{overpic}
      [width=12cm, 
      tics=20]{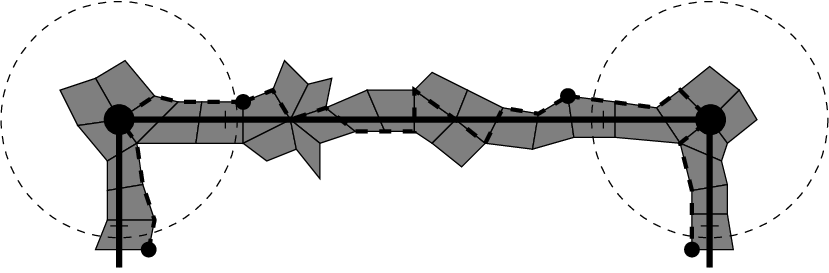}
      \put(14,33.5){${\scriptstyle B(p_i, \delta/2)}$}
      \put(84,33.5){${\scriptstyle B(p_{i+1}, \delta/2)}$}
      \put(8,15){${\scriptstyle p_i}$}
      \put(89,15){${\scriptstyle p_{i+1}}$}
      \put(50,33){${\scriptstyle \CC_i}$}
      \put(29,22){${\scriptstyle v_i}$}
      \put(68,23){${\scriptstyle v_i'}$}
      \put(19.5,2){${\scriptstyle v_{i-1}'}$}
      \put(80.5,0){${\scriptstyle v_{i+1}}$}
      \put(9,25){${\scriptstyle A_i}$}
      \put(89,24){${\scriptstyle A_{i+1}}$}
      \put(22,21.6){${\scriptstyle a_i}$}
      \put(74,21.4){${\scriptstyle a_{i+1}}$}
      \put(46,15){${\scriptstyle c_i}$}
      \put(42,23){${\scriptstyle C_i}$}
    \end{overpic}
    \caption{Construction of the curve $\CC'$.}
    \label{fig:pf_des_Grauens}
  \end{figure}
}

For $i=1, \dots, k$ we consider the graphs $G_{\mathbf{A}_i}$ and  $G_{\mathbf{C}_i}$ associated with the tile sets $\mathbf{A}_i$ and $\mathbf{C}_i$, respectively, as in  \eqref{eq:def_GM}. Each of these graphs is a union of edges in $\DD$. Moreover, $G_{\mathbf{A}_i}\sub A_i$ and 
$G_{\mathbf{C}_i}\sub C_i$. 
 Note that there is at least one tile contained
in both $\mathbf{A}_i$ and $\mathbf{C}_i$, namely any tile
containing the
common endpoint of  $\alpha_i$ and $\gamma_i$. Hence 
$G_{\mathbf{C}_i}$ and  $G_{\mathbf{A}_i}$ have a common vertex
contained in $A_i$; similarly $G_{\mathbf{C}_i}$ and
$G_{\mathbf{A}_{i+1}}$ have a common vertex contained in
$A_{i+1}$. 
 
 It follows from  Lemmas~\ref{lem:echain} and~\ref{lem:paththrough}
 that $G_{\mathbf{C}_i}$ is connected. 
Hence we can find  a simple edge path $c'_i$ in $G_{\mathbf{C}_i}$ joining a vertex $v_i\in  A_i$ as the initial point to a vertex  
$v'_{i} \in A_{i+1}$ as the terminal point. Let $c_i\coloneqq |c'_i|\sub C_i$
be the underlying arc. By deleting edges from 
$c'_i$ if necessary,  
we may assume that $v_i$ is the only vertex in $c_i \cap A_i $ 
and $v'_i$ is the only vertex in  $ c_i \cap A_{i+1}$. 
Then $c_i$  has no  other points in common
with $A_i$  or  $A_{i+1}$. 

To see this, suppose 
that there exists a point $x\ne v_i, v'_i$ with 
$x\in c_i\cap (A_i\cup A_{i+1})$,
say  $x\in c_i\cap A_i$. Then $x$ is contained in an edge $e$ of the edge path  $c'_i$. The point $x\ne v_i$ cannot be a vertex, because $v_i$ is the only vertex in $c_i\cap A_i$. So $x\in\inte(e)\cap A_i$ which implies that 
$e\sub A_i$; but then both endpoints of $e$ are vertices in $c_i\cap A_i$, which is impossible by our choice of $c'_i$.

Note that $v'_{i-1}\in C_{i-1}$ and $v_i\in C_i$ are distinct vertices in $A_i$, 
and   
recall  that $p_i\in A_i$.
Then Lemmas~\ref{lem:echain} and~\ref{lem:paththrough} imply that 
 there exists an arc $a_i\sub A_i$ with
$p_i\in a_i$ that  
consists of edges and has the endpoints  $v'_{i-1}$ and $v_i$.
Since $p_i\notin C_{i-1}\cup  C_i$, we  have $v'_{i-1}, v_i\ne  p_i$, and so $p_i\in \inte(a_i)$. 

If we arrange the arcs $a_i$ and $c_i$ in cyclic order 
$$ a_1, c_1, a_2,  c_2\dots, a_k, c_k, a_1,$$
then two of these arcs have non-empty intersection
if and only if they are adjacent in this order. If two arcs are
adjacent, then their intersection consists of a common endpoint.   
Therefore,  the set 
$$\CC'\coloneqq a_1\cup c_1 \cup a_2\cup c_2 \cup\dots \cup a_k \cup c_k$$ 
is a Jordan curve that passes through the
points $p_1, \dots, p_k$.  Moreover, $\CC'$ consists of edges and is hence contained in the $1$-skeleton $E$ of $\DD$. 

By construction each vertex  $p_i$ is an interior point of the arc
$a_i$. Thus it divides $a_i$ into two  subarcs $a_i^-$ and $ a_i^+$
consisting of edges such that $p_i$ is a common endpoint of $a_i^-$ and $ a_i^+$,   and such that  $a^-_i$ shares an endpoint  with $c_{i-1}$ and $ a_i^+$ one with $c_i$. Then $$\CC'_i\coloneqq a^+_i\cup c_i\cup a^-_{i+1}$$ for $i=1,
\dots, k$ is an arc that consists of edges and has  endpoints $p_i$ and
$p_{i+1}$. The arcs $\CC'_1, \dots,  \CC'_k$ have pairwise disjoint interior.  Moreover, 
$${\CC}'=\CC'_1 \cup \dots \cup \CC'_k.  $$ The arc $\CC'_i$ 
has the endpoints $p_i,p_{i+1}\in P$, but contains no other points in $P$. So $\inte(\CC'_i)\sub \CC'\setminus P$, 
and by \eqref{eq:ABCDinNd} we have 
$$  \CC'_i\sub A_{i}\cup C_i \cup A_{i+1}
   \sub  
  \mathcal{N}_{\delta}(\CC_i). $$
Hence by Proposition~\ref{prop:isotreln} and choice of $\delta$, the curve $\CC'$ is isotopic to $\CC$ rel.\ $P$. 

It remains to show that no tile in $\DD$ joins opposite sides of  $(\CC', P)$.
To see this, we argue by contradiction. Suppose that there exists a tile 
 $X$ in $\DD$ that joins opposite 
sides of $(\CC',P)$. Then $K\coloneqq  \mathcal{N}_\delta(X)$  joins opposite sides of $(\CC,P)$, since $\CC'_i\sub \mathcal{N}_\delta(\CC_i)$ for all $i=1, \dots, k$.     
By choice of $\delta_0$ we then have 
$$ \delta_0\le \diam(K)\le  2\delta+\diam(X)\le 2\delta+\eps_0<3\delta<\delta_0,$$
which is impossible.
\end{proof} 

\ifthenelse{\boolean{singlechapter}}{

%
%


\chapter{Subdivisions}
\label{cha:subdivisions}

In complex dynamics the iteration of polynomials is much better
understood than the iteration of general  rational maps.  One of the reasons 
 is that for
polynomials  powerful combinatorial methods  are available 
such as   external rays, Hubbard trees, or Yoccoz puzzles (see
\cite{DH84}). It is desirable to develop similar concepts for other
classes of maps as well. For Thurston maps we will introduce the  notion of a {\em two-tile subdivision rule} in this chapter. 
 It provides a useful combinatorial tool for their investigation. 

This concept can be extracted from  various previous examples 
(see Sections~\ref{sec:Lattes} and~\ref{sec:int-frac-sph}, or 
 Examples~\ref{ex:tringflP} and~\ref{ex:no_per_crit_not_exp}), where we have described Thurston maps by a
subdivision procedure. In these examples  we consider a topological $2$-sphere 
 obtained as a pillow (see Section~\ref{sec:expratThmaps})
by gluing two
$k$-gons together  along their boundaries. Then the two faces of the pillow (the $0$-tiles) are  subdivided
into $k$-gons (the $1$-tiles) and it is specified how  the map sends  a $1$-tile 
to  one of the $0$-tiles. The equator of the pillow is a  Jordan  curve that is invariant under the map and contains its postcritical points.

More generally, let $f\:S^2 \ra S^2$ be a Thurston map with
$\# \post(f)\ge 3$, and $\CC\sub S^2$ be a Jordan curve with $\post(f)\sub
\CC$.  
If $\CC$ is $f$-invariant (i.e., $f(\CC)\sub \CC$), then the
cell decompositions $\DD^n=\DD^n(f,\CC)$ 
(see Definition~\ref{def:DDn})
have nice compatibility properties given by 
Proposition~\ref{prop:invmarkov}. In particular, $\DD^{n+k}$ is a
refinement of $\DD^n$, whenever $n,k\in \N_0$. Intuitively, this means that each cell  $\DD^n$ is ``subdivided'' by the cells in
$\DD^{n+k}$. A cell  $c\in \DD^n$ is actually  subdivided by the cells in  
$\DD^{n+k}$ ``in the same way'' as the cell $f^n(c)\in \DD^0$ by the cells in $\DD^k$ (see Proposition~\ref{prop:invmarkov}~\ref{item:invmarkov5} 
for a precise statement).  
 This implies  that the 
``combinatorics'' 
of the sequence $\DD^0, \DD^1,\DD^2,\dots$ 
is uniquely determined by the pair
$(\DD^1, \DD^0)$ and the map $\tau \in \DD^1\ra
f(\tau) \in \DD^0$, i.e., the labeling $L\: \DD^1\ra \DD^0$
induced by $f$ (see Section~\ref{sec:labelings}).
For more discussion see Remark~\ref{rem:two-tile_fV0}~(ii) and the related Proposition~\ref{prop:subandD^n}. 

The triples $(\DD^1, \DD^0,
L)$ arising in this way lead to the following definition (see the
beginning of Section~\ref{sec:subdivisions} for more motivation).

\begin{definition}[Two-tile subdivision
  rules]\label{def:subdivcomb} 
Let 
$S^2$ be a $2$-sphere. 
A {\em two-tile subdivision rule}\index{two-tile subdivision rule|textbf}\index{subdivision}
  for  $S^2$ is a triple 
 $(\DD^1, \DD^0, L)$ of cell decompositions $\DD^0$ and $\DD^1$ of  $S^2$
 and an orientation-preserving  labeling $L\: \DD^1\ra \DD^0$. We assume that the cell decompositions 
satisfy the following conditions:
 
  \begin{enumerate}
  
  \item
    \label{item:subdivcomb1} 
    $\DD^0$ contains precisely two tiles.    
 
  \item
    \label{item:subdivcomb2} 
    $\DD^1$ is a refinement of $\DD^0$, and $\DD^1$ contains more than
    two tiles.   
  
  \item
    \label{item:subdivcomb3} 
    If $k$ is the number of vertices in $\DD^0$, then $k\ge 3$  and every
    tile in $\DD^1$ is a $k$-gon.   
   
  \item
    \label{item:subdivcomb4} 
     Every vertex in $\DD^1$ is contained in an even number of tiles in $\DD^1$.
  \end{enumerate}  
\end{definition}

If  $\DD^0$ is  a cell decomposition of $S^2$ with precisely two tiles
$X$ and $Y$, then necessarily $\partial X= \partial Y$. The set
$\CC\coloneqq \partial X= \partial Y$ is a Jordan curve which we
call 
{\em the Jordan curve of $\DD^0$}.\index{Jordan curve!of $\DD^0$}
Then $\CC$ is the $1$-skeleton of $\DD^0$ and all vertices and edges of $\DD^0$ lie on $\CC$.
If  $k$ is the number of these vertices on $\CC$ and $\DD^1$ is another cell decomposition of $S^2$,  
then   
a Thurston map $f$ that is cellular 
for $(\DD^1, \DD^0)$  can only exist if each   tile in $\DD^1$ is a  $k$-gon, 
i.e., it contains exactly $k$ vertices and edges in its boundary. Since $f$ is not a homeomorphism,   $\DD^1$ contains more than two tiles. The number of tiles in $\DD^1$ that contain a given vertex $v$ in $\DD^1$ is equal to the length of  the cycle of $v$ in $\DD^1$.
This number has to be  
even, because it must be an integer multiple of the length 
of a vertex cycle in $\DD^0$ which is always equal to $2$.  This
motivated the requirements   \ref{item:subdivcomb2}--\ref{item:subdivcomb4}
in Definition~\ref{def:subdivcomb}.

We say that a continuous  map $f\: S^2\ra S^2$ 
{\em realizes}\index{two-tile subdivision rule!realization of}\index{map!realizing!subdivision}\index{realizing!subdivision}
 the two-tile subdivision rule  $(\DD^1, \DD^0, L)$ if $f$ is cellular for $(\DD^1, \DD^0)$ and  $f(\tau)=L(\tau)$ for each $\tau\in \DD^1$. Note that in this case $(\DD^1, \DD^0)$ is a cellular 
 Markov partition for $f$ (see Definition~\ref{def:cellular}).  


Two-tile  subdivision rules arise from Thurston maps with
invariant curves, as the following proposition shows. 

\index{f-invariant@$f$-invariant!Jordan curve}\index{Jordan curve!f-invariant@$f$-invariant}\index{invariant!Jordan curve}
\begin{prop}[Two-tile subdivision rules via Thurston maps]
  \label{prop:ThmapSub} Suppose   $f\:S^2 \ra S^2$ is     a Thurston map with $\#\post(f)\ge 3$, and   $\CC\sub S^2$  is an $f$-invariant  Jordan curve with $\post(f)\sub \CC$.
If we define $\DD^0=\DD^0(f,\CC)$, $\DD^1=\DD^1(f,\CC)$, and $L\: \DD^1\ra \DD^0$ by setting $L(\tau)=f(\tau)$ for $\tau \in \DD^1$, then $(\DD^1, \DD^0, L)$ is a  two-tile subdivision rule realized by $f$. 
\end{prop}

  Theorem~\ref{thm:main} implies that every expanding Thurston map $f$  has an iterate $F=f^n$ that realizes  a two-tile subdivision rule.

Conversely, a two-tile subdivision rule gives rise to a Thurston map with an invariant curve and this map is unique up to Thurston equivalence.

\begin{prop}
  [Thurston maps via two-tile subdivision rules]
   \label{prop:rulemapex} 
Suppose $(\DD^1,\DD^0, L)$ is a two-tile subdivision rule on
$S^2$.
Then there exists a  Thurston map $f\:S^2\ra S^2$ that realizes $(\DD^1,\DD^0, L)$.  The map $f$ is unique up to Thurston equivalence. Moreover, the Jordan curve $\CC$ of $\DD^0$ is 
 $f$-invariant and contains the set $\post(f)$.   
 \end{prop}
 
When we constructed or described certain
Thurston maps, we used these  propositions informally several times
before
(see, for example, Figures~\ref{fig:mapg}, \ref{fig:1flap_both},
\ref{fig:triangle_flap0}, \ref{fig:obstr_map},
\ref{fig:lattes244a}, \ref{fig:Lattes333}, \ref{fig:lattes236}, \ref{fig:Levy_cycle}, and \ref{fig:chebyshev}). 
In these examples a geometric picture  represented the
cell decompositions and the labeling of a two-tile subdivision
rule. Formally, we obtained the corresponding map from
Proposition~\ref{prop:rulemapex}.

Our concept of a two-tile subdivision rule is inspired
by the more general concept of  a \defn{subdivision rule} as
introduced by Cannon, Floyd, and Parry (see \cite{CFP01, CFP06,
  CFKP}, and also  \cite{BoSt, Me02}). In their definition an 
  explicit map from the $1$-cells to
$0$-cells is specified (corresponding to the map $f$ in our case); in contrast,  our definition  is purely combinatorial. 

The reason for the name {\em two-tile} subdivision rule is that 
the data given by $(\DD^1,\DD^0)$ determines how the {\em two} $0$-{\em tiles} are subdivided by the cells in $\DD^1$,  and this together with the labeling $L$ can be used to create  a sequence of cell decompositions $\DD^n$  where each cell $\tau\in \DD^1$ is subdivided by the cells in $\DD^2$ in the same way
as the cell $L(\tau)\in \DD^0$ is subdivided by the $1$-cells, etc. 
Our  definition is tailored to generate Thurston maps, so a more accurate term would have been a ``two-tile subdivision rule generating a Thurston map'', but we chose the shorter term for brevity.

Since we are mostly interested in expanding Thurston maps, we want to
find  a combinatorial condition on a two-tile subdivision rule  that ensures that it can be  realized by an expanding Thurston map. 
  To motivate  a relevant definition, suppose that    $f\: S^2 \ra S^2$ is   a  Thurston map and $\CC \sub S^2$ is a Jordan curve with $\post(f)\sub \CC$. We consider the quantities  $D_n=D_n(f,\CC)$ defined  in \eqref{def:dk}. 
      If the Jordan curve $\CC$ here is  $f$-invariant,  then it is easy to see that the numbers $D_n$ are non-decreasing as $n\to \infty$. 
We will show that we actually have an exponential increase  
 under the additional assumption that there exists $n_0\in \N$ with $D_{n_0}\ge 2$ (see Lemma~\ref{lem:submultexp}). 
  This  will turn out to be a key 
condition related to expansion of Thurston maps realizing two-tile subdivision rules. 


\index{Thurston map!combinatorially expanding|textbf}
\index{combinatorially expanding|textbf}
\index{expanding!combinatorially|textbf}      
\index{d0 n@$D_n$}  
\begin{definition}[Combinatorial expansion]
  \label{def:combexp}
  Let $f\:S^2\ra S^2$ be a Thurston map.  We call $f$
  \defn{combinatorially expanding} if $\#\post(f)\ge 3$, and if
  there exists a Jordan curve $\CC\sub S^2$ that is
  $f$-invariant, satisfies $\post(f)\sub \CC$, and for which
  there is a number $n_0\in \N$ such that $D_{n_0}(f,\CC)\ge 2$.
\end{definition} 
The condition $D_{n_0}(f,\CC)\ge 2$  means that no single $n_0$-tile for $(f,\CC)$ joins opposite sides of 
$\CC$.  

If $f$ and $\CC$ are  as in the previous definition, then we say that 
$f$ is {\em  combinatorially expanding for $\CC$}. This   condition  is indeed  combinatorial in nature, because it can be verified just by knowing the combinatorics  of the cell decompositions 
 $\DD^n=\DD^n(f,\CC)$, $n\in \N_0$.   This  in turn is  determined by the combinatorics  of  
the pair $(\DD^1, \DD^0)$ and the labeling  $\tau\in \DD^1\mapsto
f(\tau)\in \DD^0$ induced by $f$ (see
Remark~\ref{rem:two-tile_fV0}~(ii) and Proposition~\ref{prop:subandD^n}).

If a  Thurston map $f\: S^2\ra S^2$  is expanding and $\CC\sub S^2$ is an $f$-invariant Jordan curve with $\post(f)\sub \CC$, then $f$  is also combinatorially expanding for $\CC$, because in this case    $D_n(f,\CC)\to \infty$  as $n\to \infty$ (see Lemma~\ref{lem:Dtoinfty}). The converse is not true in general,  as  a combinatorially expanding Thurston map need not be expanding
(see Example~\ref{ex:barycentric}). However,  in Chapter~\ref{cha:combexp} we will see  that each combinatorially expanding Thurston map is  equivalent to 
an expanding Thurston map with an invariant curve 
(Theorem~\ref{thm:combexp1}). 

Let $(\DD^1, \DD^0, L)$ be a two-tile subdivision rule and $\CC$ be the Jordan curve of $\DD^0$. We will show that if a Thurston map realizing 
 this subdivision rule is combinatorially expanding for  $\CC$, then this is true for every Thurston map realizing the subdivision rule (Lemma~\ref{lem:invrealize}).
In this case, we say that the subdivision rule is 
{\em combinatorially expanding}  
(see Definition~\ref{def:combexprule}).\index{combinatorially expanding!two-tile subdivision rule}\index{two-tile subdivision rule!combinatorially expanding}\index{expanding!combinatorially}
We will later prove that under an additional hypothesis a two-tile subdivision rule is 
combinatorially expanding if and only if it can be realized by an expanding Thurston map (Theorem~\ref{thm:combexp2}). 

 Our
definition of combinatorial expansion is  set up to be compatible with  the description of a Thurston
map by a two-tile subdivision rule. It 
 is clearly  not 
invariant under Thurston equivalence, because  we require the
existence of an invariant Jordan curve. For  Thurston maps $f$ with an
invariant curve,  combinatorial expansion is sufficient for
$f$ to be equivalent to an expanding Thurston map 
(Theorem~\ref{thm:combexp1}). However, this condition  not necessary
(Example~\ref{ex:exp_notcexp}). We are not aware of a necessary and
sufficient condition that is easy to check in practice (see \cite[Theorem~1.4]{HP_alg_exp} for an algebraic condition for  Thurston maps without 
periodic critical points).

This chapter is organized as follows. In
Section~\ref{sec:Thurtoncurves} we summarize facts related to
Thurston maps with invariant curves and the associated cell
decompositions (Proposition~\ref{prop:invmarkov}). Then we give
yet another characterization when $f$ is expanding
(Lemma~\ref{lem:charexpint}). We show that the quantities
$D_n(f,\CC)$ are supermultiplicative (Lemma~\ref{lem:submult}).

 In Section~\ref{sec:subdivisions} we discuss two-tile subdivision rules and prove Propositions~\ref{prop:ThmapSub} and ~\ref{prop:rulemapex}.  For two-tile subdivision rules the information given by an orienta\-tion-preserving labeling  
 $L\: \DD^1\ra \DD^0$ can be further compressed: for example, it is  uniquely determined if one knows the image of one flag in $\DD^1$  (see Lemma~\ref{lem:labeluniq} which is based on Lemma~\ref{lem:labelexis}). 
We will also discuss facts related to 
combinatorial expansion 
 (such as Lemma~\ref{lem:invrealize}) and conclude  the section    with a precise version of the  statement that the combinatorics of the sequence of cell decompositions $\DD^n$
 of a Thurston map realizing a subdivision rule  is determined by the subdivision rule alone (Proposition~\ref{prop:subandD^n}). 

Our results  pave the way for a convenient construction of Thurston maps from a combinatorial perspective. Section~\ref{sec:examples-two-tile} is devoted to this. We will exhibit several Thurston maps arising  from two-tile subdivision rules.

\section{Thurston maps with invariant curves}
\label{sec:Thurtoncurves}

 In the following, $f\: S^2\ra S^2$ is  a Thurston map, $\CC\sub S^2$ is  a Jordan curve with  $\post(f)\sub \CC$, and 
$\DD^n=\DD^n(f,\CC)$ for $n\in \N_0$ is  the cell decomposition
of $S^2$ given by the $n$-cells for $(f,\CC)$ according to
Definition~\ref{def:DDn}. 

As usual, a set  $M\sub S^2$ is called 
{\em $f$-invariant}
\index{invariant!set} 
\index{f-invariant@$f$-invariant}
(or simply {\em invariant} if $f$ is understood) if 
  \begin{equation}
  \label{eq:deffinv}
  f(M)\subset M\quad  \text{or equivalently} \quad M \sub f^{-1}(M). 
\end{equation}
We will mostly be interested in the case when $M=\CC$ is a Jordan
curve with $\post(f)\subset \CC$. The reason for this is that then the   cell
decomposition $\DD^n(f,\CC)$ induced by such an 
$f$-invariant Jordan curve\index{f-invariant@$f$-invariant!Jordan curve|textbf}\index{Jordan curve!f-invariant@$f$-invariant|textbf}\index{invariant!Jordan curve|textbf}
$\CC$ is refined by each cell decomposition $\DD^{m}(f,\CC)$ of higher levels  
$m\ge n$ 
(see Proposition~\ref{prop:invmarkov} below). 

Since the set $\post(f)$ is $f$-invariant, we  have  
\begin{equation} \label {1postinv}
\post(f)\sub f^{-1}(\post(f))\sub f^{-2}(\post(f))\sub \dots\,. 
\end{equation}
 
We know by  Proposition~\ref{prop:celldecomp}~\ref{item:skeletons}
that $${\bf V}^n={\bf V}^n(f,\CC)=f^{-n}(\post(f))$$ for $n\in \N_0$, and so \eqref{1postinv} is equivalent to the inclusions  
$${\bf V}^0\sub {\bf V}^1\sub {\bf V}^2\sub \dots $$
for the vertex sets of the cell decompositions $\DD^n$. 

In general,   a similar inclusion chain will not hold  for the $1$-skeleta $E^n\coloneqq f^{-n}(\CC)$ of $\DD^n$, but if $\CC$ is $f$-invariant, then it follows by induction that  
$$\CC=E^0\sub E^1\sub E^2\sub \dots\,.$$

The following proposition summarizes the properties of 
the cell decompositions 
$\DD^n(f,\CC)$ if $\CC$ is $f$-invariant. 

\begin{prop}
  \label{prop:invmarkov} 
  Let $k,n\in \N_0$, $f\: S^2\ra S^2$ be a Thurston map, and
  $\CC\sub S^2$ be an $f$-invariant Jordan curve with
  $\post(f)\sub \CC$.  Then we have:
  \index{f-invariant@$f$-invariant!Jordan curve}\index{Jordan curve!f-invariant@$f$-invariant}\index{invariant!Jordan curve}

\begin{enumerate}

\item
  \label{item:invmarkov1} 
  $\DD^{n+k}$ is a refinement of $\DD^k$,
  and 
  $(\DD^{n+k}, \DD^k)$ is a cellular Mar\-kov partition for $f^n$.  

\item
  \label{item:invmarkov2}
  Every
  $(n+k)$-tile $X^{n+k}$ is contained in a unique $k$-tile $X^k$. 

\item
  \label{item:invmarkov3}
  Every $k$-tile $X^k$ is equal to 
  the union of all $(n+k)$-tiles $X^{n+k}$
satisfying $X^{n+k}\sub X^k$.

\item
  \label{item:invmarkov4}
  Every $k$-edge  $e^k$ is equal to 
  the union of all $(n+k)$-edges 
 $e^{n+k}$
satsifying
$e^{n+k}\sub e^k$.

\item
  \label{item:invmarkov5} Let $c'\sub S^2$ be an $n$-cell  and
  $c\coloneqq f^n(c')$. Define 
  \begin{align*}
    \mathcal{M}'
    &\coloneqq
    \{ \tau': \tau' \text{ is an $(n+k)$-cell with $\tau'\sub c'$}
      \} 
      \text{ and}
    \\
    \mathcal{M}
    &\coloneqq
    \{ \tau: \tau \text{ is a $k$-cell with $\tau\sub c$} \}.     
  \end{align*}
Then $\mathcal{M}'$ and $\mathcal{M}$ are cell decompositions of $c'$ and $c$, respectively, and 
the map $\tau'\in \mathcal{M}'  \mapsto f^n(\tau')$ is an isomorphism of the cell complexes $\mathcal{M}'$ and $\mathcal{M}$.
\end{enumerate}

\end{prop}

If $f$ and $\CC$ are as in this proposition, then, in particular,
the pair $(\DD^1, \DD^0)$ is a cellular Markov partition for $f$
by statement \ref{item:invmarkov1}.
If $X^n $ is any $n$-tile, then by \ref{item:invmarkov2} there exist unique
$i$-tiles $X^i$ for $i=0, \dots, n-1$ such that 
$$X^n\sub X^{n-1}\sub \dots\sub X^0. $$ 

We refer to the statements \ref{item:invmarkov3} and
\ref{item:invmarkov4} informally by saying that  
 tiles and edges  are ``subdivided'' by tiles and edges of higher levels.
By statement  \ref{item:invmarkov5} each $n$-cell $c'$ is subdivided 
 by the $(n+k)$-cells contained in $c'$ ``in the same way'' as the corresponding $0$-cell $c=f^n(c')$ is subdivided by the $k$-cells  
 contained in $c$. 
 
\begin{proof}
  \ref{item:invmarkov1} 
  We know that the map $f^n$ is cellular for
  $(\DD^{n+k}, \DD^n)$
  (Proposition~\ref{prop:celldecomp}~\ref{item:fkcellular}); so we
  have 
  to show that $\DD^{n+k}$ is a refinement of $\DD^n$ (see
  Definition~\ref{def:ref}).  Since  $\CC$ is $f$-invariant, we have
  $E^{n+k}=f^{-(n+k)}(\CC)\supset E^k=f^{-k}(\CC)$, and so
  $S^2\setminus E^{n+k}\sub S^2\setminus E^k$.

To establish the first property of a refinement, we will show  
   that every $(n+k)$-cell is contained in some $k$-tile. 

 Let $\sigma$ be an arbitrary $(n+k)$-cell. If $\sigma$ is an
  $(n+k)$-tile, then $\inte(\sigma)$ is a connected set in 
$S^2\setminus E^{n+k}\sub S^2\setminus E^k$
and hence contained in the interior of a  $k$-tile $\tau$ (see
Proposition~\ref{prop:celldecomp}~\ref{item:nedgesC}).  
It follows that  $\sigma=\overline {\inte(\sigma)}\sub \tau$. 

 If $\sigma $ is an $(n+k)$-edge or an $(n+k)$-vertex, then it is contained in an
   $(n+k)$-tile (Lemma~\ref{lem:specprop}~\ref{item:prop_cell4}
   and~\ref{item:prop_cell5}), and hence in some  
   $k$-tile  by what we have just seen.

To establish the second property of a refinement,  let $\tau$ be an arbitrary $k$-cell.  We have  to show  that 
the $(n+k)$-cells $\sigma$ contained in $\tau$  cover $\tau$.  

If $\tau$ consists of a $k$-vertex $p$,  then $p$ is also an $(n+k)$-vertex, and the statement is trivial.

 If $\tau$ is a $k$-edge, consider the points in ${\bf V}^{n+k}$ that lie on $\tau$.
 Note that this includes the elements of $\partial \tau \sub {\bf V}^k\sub {\bf V}^{n+k}$. 
 By using these points to partition $\tau$, we can find  finitely many arcs $\alpha_1, \dots, \alpha_N$ such that
$\tau=\alpha_1\cup\dots\cup  \alpha_N$ and such that each arc $\alpha_i$ has endpoints in ${\bf V}^{n+k}\supset {\bf V}^k$  and  interior $\inte(\alpha_i)$ disjoint from 
${\bf V}^{n+k}$. Then  for each $i=1, \dots, N$ the set $\inte(\alpha_i)$ is a connected set in 
$E^k\setminus {\bf V}^{n+k}\sub E^{n+k}\setminus {\bf V}^{n+k}$. It follows that $\inte(\alpha_i)$ and hence also  $\alpha_i$ is contained in some  $(n+k)$-edge $\sigma_i$ (Proposition~\ref{prop:celldecomp}~\ref{item:nedgesC}).  
Since the endpoints of $\alpha_i$ lie in ${\bf V}^{n+k}$, they cannot lie in $\inte(\sigma_i)$, and so they are also endpoints of $\sigma_i$. This implies that $\alpha_i=\sigma_i$. In particular, the $(n+k)$-edges $\sigma_1, \dots ,\sigma_N$ are contained in $\tau$ and form a cover of $\tau$. The statement follows in this case. 
 
 Finally, let  $\tau$ be a $k$-tile. If $p\in\inte(\tau)$ is arbitrary, then $p$ is contained in an $(n+k)$-tile $\sigma$. 
 By the first part of the proof, $\sigma$ is contained in a $k$-tile. Since $\tau$ is the only $k$-cell that contains $p$, we must have $\sigma= \tau$. 
 This implies that the union of the $(n+k)$-tiles  contained in $\tau$ covers $\inte(\tau)$. On the other hand, this union consists of finitely many tiles and is hence a closed set. 
 It follows that the union also contains
 $\overline{\inte(\tau)}=\tau$.

 \smallskip 
 \ref{item:invmarkov2}
 We have just seen  that every $(n+k)$-tile $X^{n+k}$ is
 contained in a $k$-tile $X^k$. This tile is unique. To see this,
 suppose  that $\widetilde X^k$ is another $k$-tile with $X^{n+k}\sub \widetilde X^k$.
Then
$$ \emptyset\ne \inte(X^{n+k})\sub \inte(X^k)\cap \inte(\widetilde X^k),$$
and so $X^k$ and $\widetilde X^k$ have common interior points. This implies 
$X^k=\widetilde X^k$. 

\smallskip
\ref{item:invmarkov3}--\ref{item:invmarkov4} Both statements were established in the proof of \ref{item:invmarkov1}. 

\smallskip
\ref{item:invmarkov5} Note that under the given assumptions it  follows from  \ref{item:invmarkov1} that 
$c=f^n(c')$ is a $0$-cell. Moreover, again by  \ref{item:invmarkov1} 
the cell decomposition $\DD^{n+k}$ is a refinement of $\DD^n$, and $\DD^k$ is a refinement of $\DD^0$. This implies that  $\mathcal{M}'$ and $\mathcal{M}$ are cell decompositions of  $c'$ and $c$, respectively. 

It also follows from  \ref{item:invmarkov1}  that  
$f^n(\tau')\in \mathcal{M}$ whenever $\tau'\in \mathcal{M}'$. So we can define a map  $\varphi\: 
\mathcal{M}'\ra \mathcal{M}$ by setting $ \varphi(\tau')=f^n(\tau)$ for 
$\tau'\in \mathcal{M}'$. 
We have to show that this  map   $\varphi$ is an isomorphism of cell 
complexes (see Definition~\ref{def:compiso}). 

The map   $f^n|c'$ is a homeomorphism of the $n$-cell $c'$ onto the $0$-cell $c$. This implies that $\varphi$ is injective and that it satisfies the   
 conditions~\ref{item:compiso1} and~\ref{item:compiso2} in 
 Definition~\ref{def:compiso}.

 It remains to show that $\varphi$ is also surjective. To see this, let $\tau\in   \mathcal{M}$ be an arbitrary $k$-cell with $\tau\sub c$. Since $f^n|c'$ is a homeomorphism 
of $c'$ onto $c$, the set $\tau'\coloneqq (f^n|c')^{-1}(\tau)\sub c'$ is a topological cell. Moreover, $f^n|\tau'$ is a homeomorphism of $\tau'$ onto the $k$-cell $\tau$. Lemma~\ref{adhoc}~\ref{item:adhoc1} now implies that $\tau'$ is an $(n+k)$-cell, and so $\tau'\in 
  \mathcal{M}'$. Then $\varphi(\tau')=f^n(\tau')=\tau$, and so $\varphi$ is indeed surjective.  
\end{proof} 

\ifthenelse{\boolean{nofigures}}{}{
\begin{figure}
  \centering
  \begin{subfigure}{0.48\textwidth}
    \begin{overpic}
      [width=2.395in, 
      tics=20]{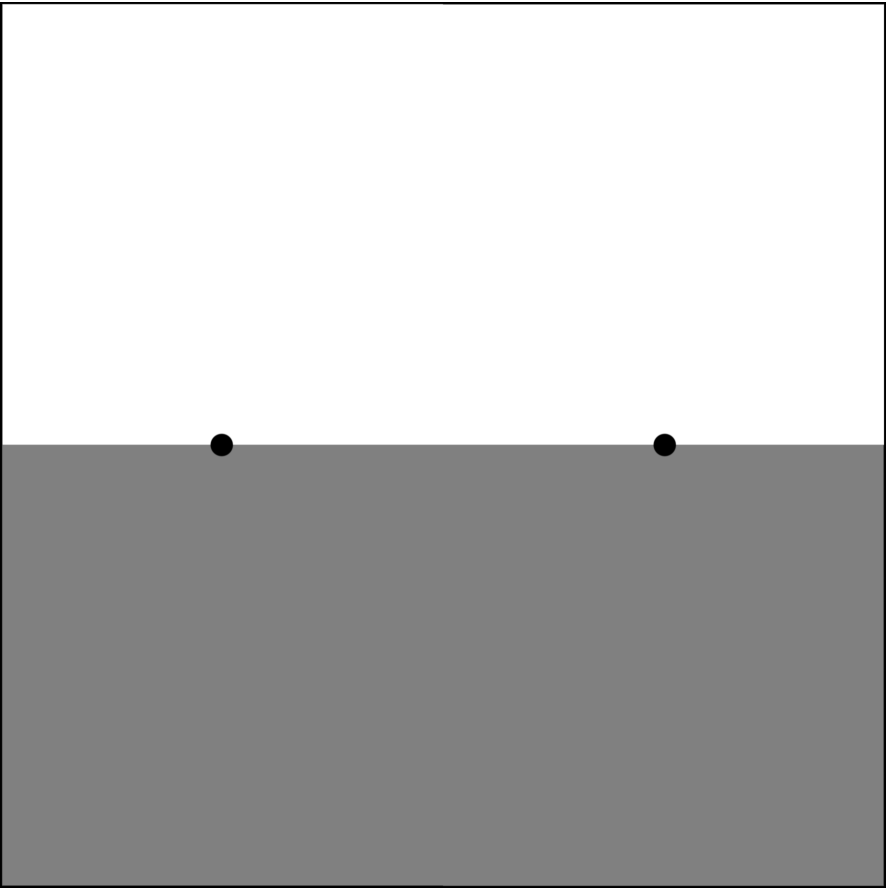}
      \put(19,53){$-1$}
      \put(73,53){$1$}
    \end{overpic}
  \end{subfigure}
  \hspace*{\fill}
  \begin{subfigure}{0.48\textwidth}
    \begin{overpic}
      [width=2.395in, 
      tics=20]{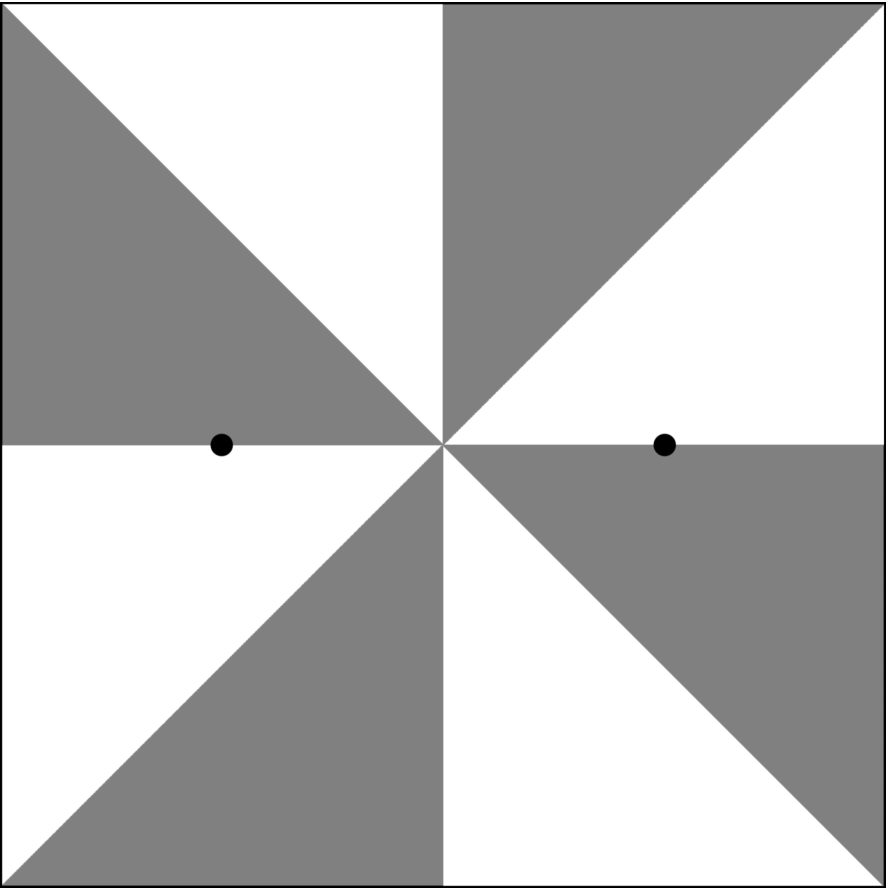}
    \end{overpic}
  \end{subfigure}

  \vspace{0.04\textwidth}
  \begin{subfigure}{0.48\textwidth}
    \begin{overpic}
      [width=2.395in, 
      tics=20]{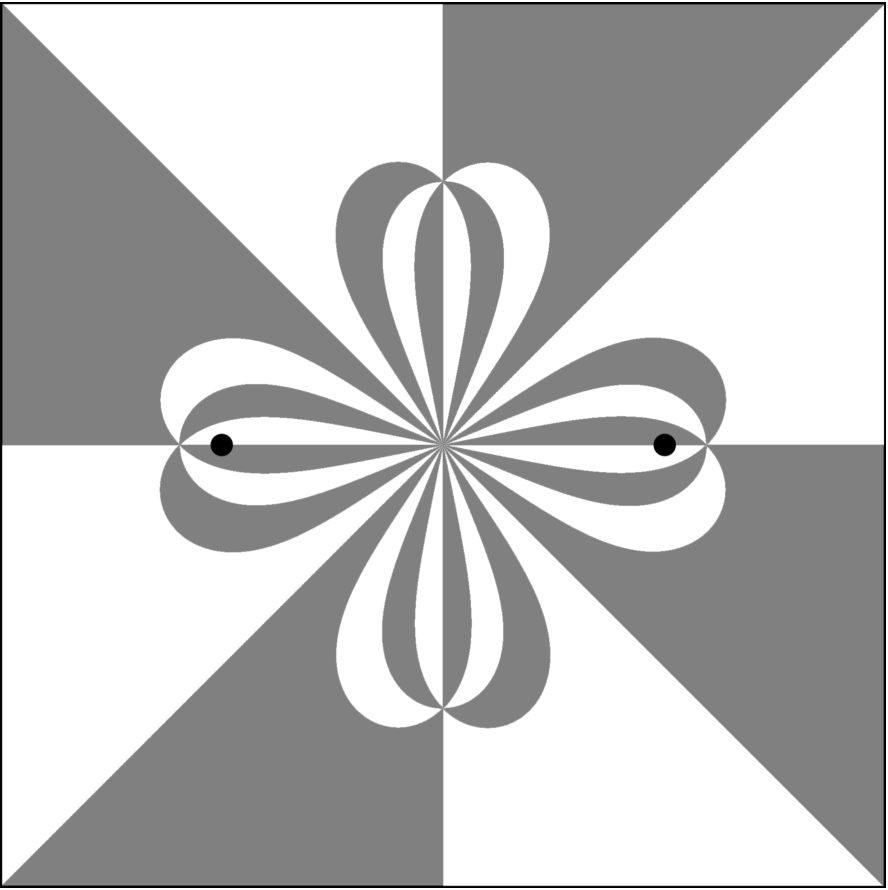}
    \end{overpic}
  \end{subfigure}
  \hspace*{\fill}
  \begin{subfigure}{0.48\textwidth}
    \begin{overpic}
      [width=2.395in, tics=20]{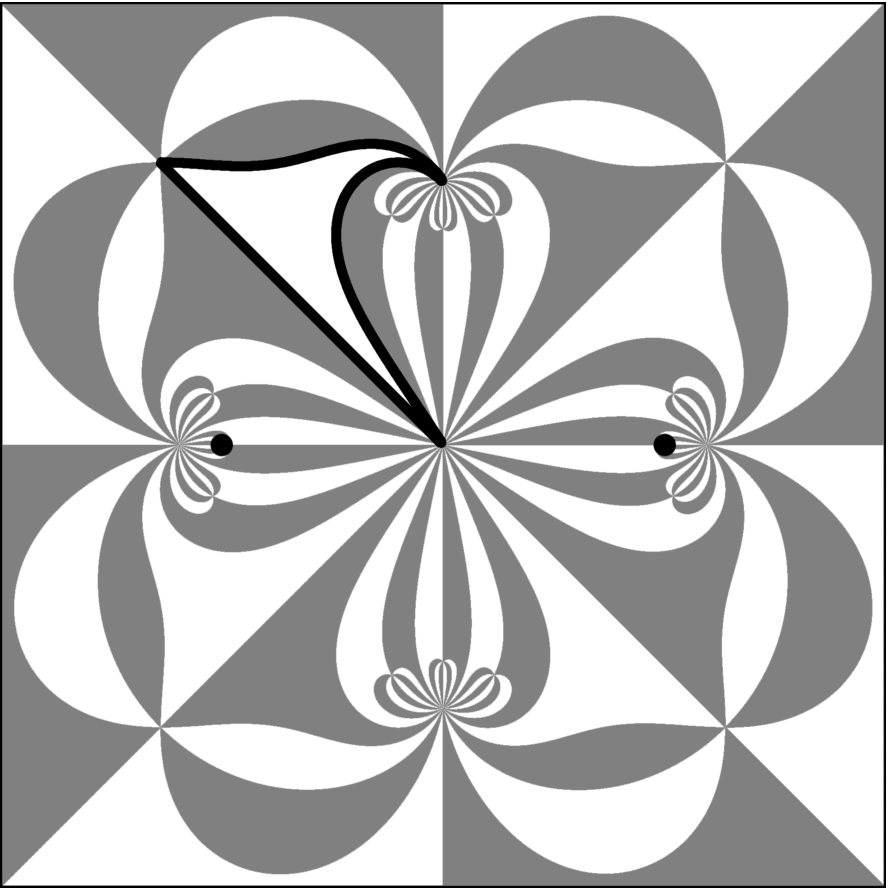}
    \end{overpic}
  \end{subfigure}

  \vspace{0.04\textwidth}
  \begin{subfigure}{0.48\textwidth}
    \begin{overpic}
      [width=2.395in, tics=20]{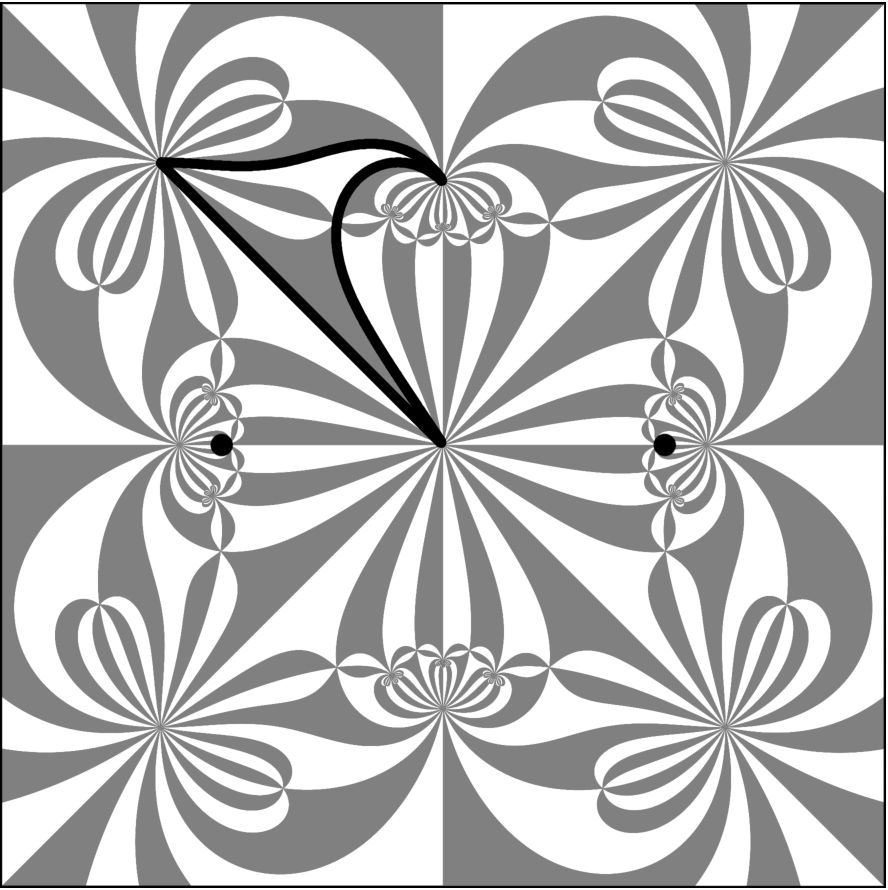}
    \end{overpic}
  \end{subfigure}
  \hspace*{\fill}
  \begin{subfigure}{0.48\textwidth}
    \begin{overpic}
      [width=2.395in, tics=20]{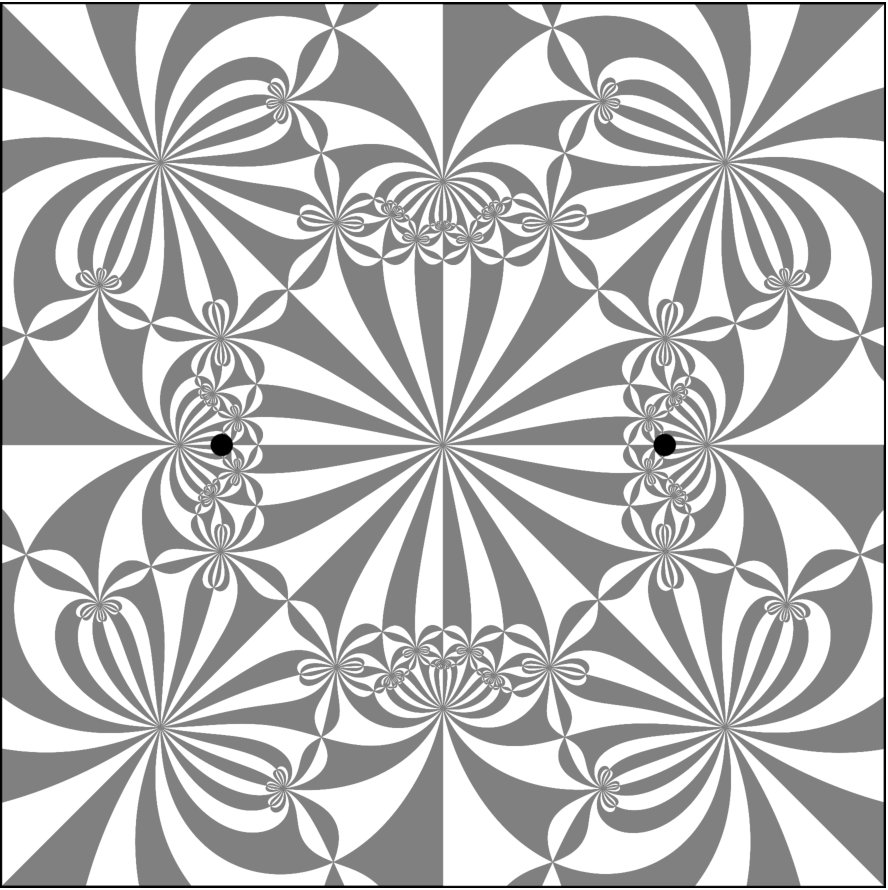}
    \end{overpic}
  \end{subfigure}
  \caption{Subdividing tiles.}
  \label{fig:sub_div_tiles}
\end{figure}

}

We illustrate the previous proposition by an example that shows  how cells are subdivided. 

\begin{ex}
  \label{ex:f2flaps}
  In Figure~\ref{fig:sub_div_tiles} we indicate the cell decompositions 
  generated  by a Thurston map with an invariant 
  Jordan curve and the corresponding subdivision of cells. 
  Here we use 
  the map $f\colon \CDach \to \CDach$ given by $f(z) =
  1-2/z^4$. The ramification portrait of $f$ is
  \begin{equation*}
    \xymatrix @R=1pt{
      0 \ar[r]^{4:1} 
      & 
      \infty  \ar[r]^{4:1} 
      & 
      1 \ar[r] 
      & 
      -1 \rlap{.}\ar@(r,u)[]
    }
  \end{equation*}
  Thus $\post(f)= \{-1,1,\infty\}$. Let
  $\CC\coloneqq \widehat{\R}$ be the extended real line. Clearly,
  $\post(f) \subset \CC$ and $f(\CC) \subset \CC$, since the
  coefficients of $f$ are real (indeed $f(\CC) = [-\infty,1]$).    
  
  We let the closure of the upper half-plane be
  the white $0$-tile, and the closure of the lower half-plane be
  the black $0$-tile. The intersections of the resulting tiles of levels $0$ to $5$  with the square $[-2,2]^2\subset \R^2\cong \C$ 
  are shown in Figure~\ref{fig:sub_div_tiles}. 
  
 The upper and lower half-planes (i.e., the two
  $0$-tiles $\XOw$ and $\XOb$  shown on the top left) 
  are both subdivided into  four  $1$-tiles (shown on the top
  right). Similarly, each $n$-tile is subdivided into four 
  $(n+1)$-tiles. For illustration we have marked the boundary of one
  white $3$-tile (shown on the middle right) and the four
  $4$-tiles into which it is subdivided (shown on the bottom
  left). 

  Similarly, the $0$-edge $[-1,1] \subset \CC= \widehat{\R}$ is
  subdivided into the two $1$-edges $[-1,0]$ and $[0,1]$. Note
  that the $0$-edges $[-\infty,-1]$ and $[1,\infty]$ are also 
  $1$-edges. So these $0$-edges are each replaced 
  with  one $1$-edge. In the same way, each $n$-edge is replaced
  with one or two $(n+1)$-edges. 

  Note that it is possible to obtain the $(n+1)$-tiles from
  the $n$-tiles in the following way. Given an $n$-tile $X$, let
  $X^0$ be the $0$-tile of the same color.  Then there exists a unique homeomorphism
 $\varphi$ from $X^0$ (the upper or lower half-plane)
  onto $X$ that is a conformal map between the interiors of these tiles and 
  sends the points $-1,1,\infty$ (the $0$-vertices) to the corresponding vertices of $X$.  The map  $\varphi$ sends the four $1$-tiles  that
  subdivide $X^0$ to $X$.  The images of these $1$-tiles are 
 the
  $(n+1)$-tiles into which $X$ is subdivided. Similarly, $\varphi$ gives a bijection 
  between the 
  $1$-edges and $1$-vertices contained in $X$ and 
   the
  $(n+1)$-edges and $(n+1)$-vertices  contained in
  $X$. Note that  $\varphi$ is determined uniquely by $X$, once 
   we know the
  color of $X$, and the correspondence 
  between the vertices of $X$ and the $0$-vertices.
  We introduced the concept of a labeling (see Section~\ref{sec:labelings}) to keep track of such information. 
\end{ex}

As before, let $f\: S^2\ra S^2$ be a Thurston map, and $\CC\sub S^2$ be an
$f$-invariant Jordan curve with $\post(f)\sub \CC$. We consider cells for $(f,\CC)$.  By 
 $\mathcal{S}=\mathcal{S}(f,\CC)$ 
\index{S..@$\mathcal{S}(f,\CC)$}
 we denote the set of all sequences $\{X^n\}$, where  $X^n$ is an
$n$-tile for  $n\in \N_0$ and 
\begin{equation}
  \label{eq:Xseq} 
  X^0\supset X^1\supset X^2\supset \dots\,.
\end{equation}
  Since tiles are subdivided by tiles of higher level,  for each point 
$p\in S^2$ we can find a sequence $\{X^n\}\in \mathcal{S}$ such that 
$p\in \bigcap_n X^n$.  Here it is understood that the intersection is  taken over  all $n\in \N_0$. In the following,  we use 
a similar convention for intersections of sets labeled by some index $n$, $k$, etc.,  if the range of the indices   is clear from the context. 

 In general, a  sequence $\{X^n\}\in \mathcal{S}$ that contains a given point $p\in S^2$ is not  unique. Moreover,  the intersection $\bigcap_n X^n$ may  contain more than
one point. It turns out that this gives a criterion when $f$ is expanding.

\begin{lemma} 
  \label{lem:charexpint}
  \index{expanding} 
  \index{Thurston map!expanding} 
  Let $f\: S^2\ra
  S^2$ be a Thurston map, and $\CC\sub S^2$ be an $f$-invariant Jordan
  curve with $\post(f)\sub \CC$. Then the map $f$ is expanding if and
  only if for each sequence $\{X^n\}\in \mathcal{S}(f,\CC)$ the
  intersection $\bigcap_n X^n$ consists of precisely one point.
\end{lemma}

\begin{proof} Fix a metric on $ S^2$ that induces the  given topology on $S^2$. If $f$ is expanding and
$\{X^n\}\in \mathcal{S}=\mathcal{S}(f,\CC)$, then 
$\diam(X^n)\to 0$ as $n\to \infty$. Hence $\bigcap_n X^n$ cannot contain more than one point. On the other hand, this set is an intersection of 
a nested sequence of 
non-empty compact sets and hence non-empty.
So the set $\bigcap_n X^n$ contains precisely  one point. 
 
For the converse direction suppose that $\bigcap_n X^n$ is a singleton set for each sequence $\{X^n\}\in \mathcal{S}$. 
To establish that 
$ f$ is expanding we have to show that 
$$\lim_{n\to\infty} \max\,\{ \diam ( X): X \text{ is an $n$-tile} \}=0.$$ 
We argue by contradiction and assume that this is not the case. Then there exists $\delta>0$ 
such that $\diam( X )\ge \delta$ for some tiles $X$   of arbitrarily high level.

We  now define a descending sequence of tiles 
$X^0\supset X^1\supset X^2\supset\dots$ as follows.  Let $X^0$ be a $0$-tile such that 
$X^0$ contains  tiles $X$ of arbitrarily high levels
with $\diam( X )\ge \delta$. Since 
 every tile is contained in one of the two  $0$-tiles, there exists such a $0$-tile $X^0$. Note that then 
$\diam( X^0 )\ge \delta$. Moreover, among the finitely many $1$-tiles into which  $X^0$ is subdivided there must be a $1$-tile $X^1\sub X^0$ such that $X^1$ 
contains tiles $X$ of arbitrarily high levels
with $\diam(X )\ge \delta$. Again this implies that $\diam( X^1 )\ge \delta$.
Repeating this procedure, we obtain a sequence $\{X^n\}\in \mathcal{S}$ such that  
$\diam( X^n )\ge \delta$ for all $n\in \N_0$.
It is easy to see that this implies that the set $\bigcap_n X^n$ also has 
diameter $\ge \delta>0$, and so it contains at least two points. This is a contradiction showing that $f$ is expanding. 
\end{proof} 

The main idea in the previous proof is essentially K\H{o}nig's
infinity lemma from graph 
theory\index{Konig's infinity lemma@K\H{o}nig's infinity lemma} (see for 
example \cite[Lemma~9.1.3]{Di});   it says that a (locally finite) simplicial  tree with arbitrarily long branches has an infinite branch.


%



 Recall the definition of the numbers $D_n=D_n(f,\CC)$ in \eqref{def:dk}. 
We know that $D_n\to \infty$ if $f$ is expanding (see Lemma~\ref{lem:Dtoinfty}). If $f$ is not necessarily expanding, but $\#\post(f)\ge 3$ and  the Jordan curve $\CC$ used in the definition of $D_n$ is  invariant,  then  the numbers $D_n$ are non-decreasing, i.e., $ D_{n+1}\ge D_n$ for  $n\in \N_0$.

To see this, we consider tiles for $(f,\CC)$.  Let   $n\in \N_0$ be arbitrary.  Then by definition of $D_{n+1}$, there exist 
$(n+1)$-tiles $X_1, \dots, X_N$ with $N=D_{n+1}$ whose union is a connected set  joining opposite sides of $\CC$. The tile $X_i$ is contained in an $n$-tile $Y_i$ for $i=1, \dots, N$. Then the union of the $n$-tiles $Y_1, \dots, Y_N$  is a connected set joining opposites sides of $\CC$. It follows that $D_n\le N=D_{n+1}$ as desired.

 The following lemma establishes 
the much deeper fact that these quantities are actually supermultiplicative (in an appropriate  sense). This  implies  that the numbers $D_n$ increase exponentially fast under the additional assumption that there exists $n_0\in \N$ with $D_{n_0}\ge 2$ (see Lemma~\ref{lem:submultexp}). 

\index{d0 n@$D_n$} 
\begin{lemma} 
  \label{lem:submult}
  Let $f\: S^2\ra S^2$ be a Thurston map with $\#\post(f)\ge 3$,    $\CC\sub S^2$ be a Jordan curve with $\post(f)\sub \CC$, and  $D_n=D_n(f,\CC)$ for $n\in \N_0$. 
 Suppose that  $\CC$ is $f$-invariant.  Then for all $n,k\in \N_0$ we have 
  $$ D_{n+k}\ge D_nD_k$$  
   if $\#\post(f)\ge 4$, and 
  \begin{equation} 
    \label{weaksuper} 
    D_{n+k}\ge D_n(D_k-1)+1
  \end{equation} 
  if  $\#\post(f)=3$. 
\end{lemma}
  
In the proof of this lemma we will use $n$-chains. Recall from
Definition~\ref{def:chains} that such an $n$-chain is a finite
sequence of $n$-tiles $X_1,\dots, 
X_N$ with  $X_i\cap X_{i+1}\neq \emptyset$ for $i=1, \dots,
N-1$.
\index{n9@$n$-!chain}



\ifthenelse{\boolean{nofigures}}{}{
\begin{figure}
  \centering
  \begin{overpic}    
    [width=10cm, 
    tics=20]{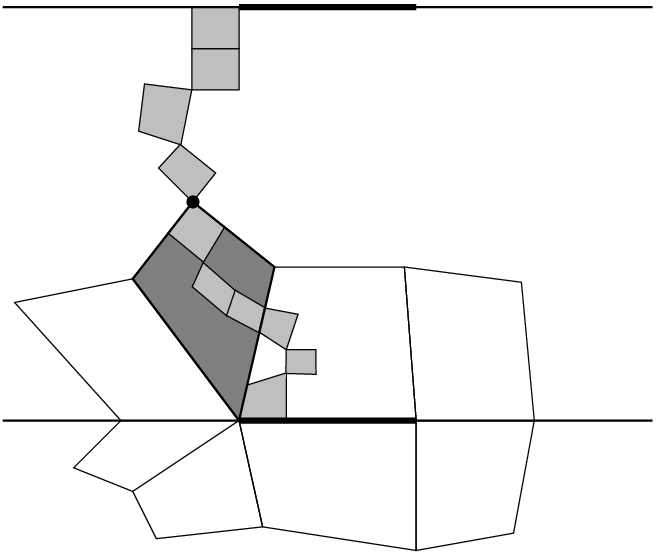}
    \put(93,15){$E$}
    \put(49,17){$e=c_1$}
    \put(44,22){${X_1}$}
    \put(93,78){$\widetilde{E}$}
    \put(44,78){$\widetilde{e}=c_{M+1}$}
    \put(70,45){$W^k(c_1)$}
    \put(24.6,53){$c_2$}
    \put(33,51.5){$X_{s_2-1}$}
    \put(31,60){$X_{s_2}$}
    \put(24,28){$Y_1$}
    \put(22,79){$X_N$}
  \end{overpic}
  \caption{The proof of Lemma \ref{lem:submult}.}  
  \label{fig:proofDk}
\end{figure}
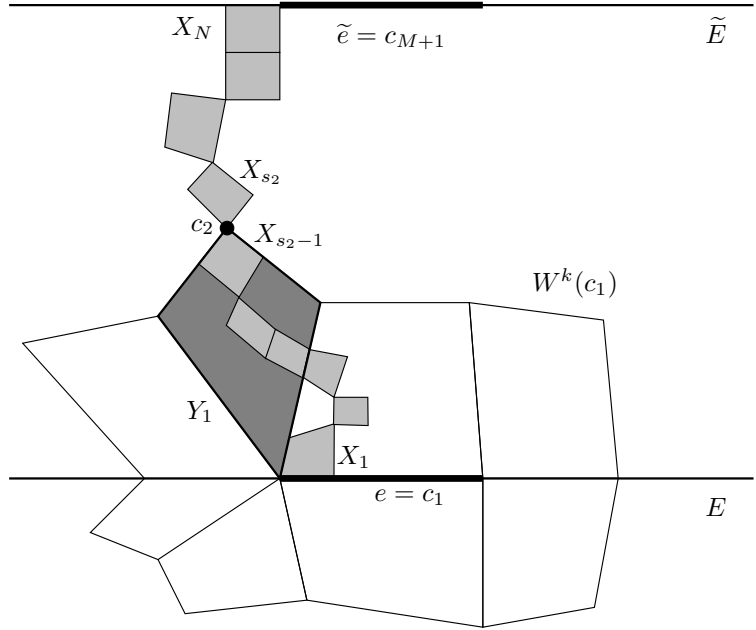
}
 
\begin{proof}[Proof of Lemma \ref{lem:submult}]
    
 {\em Case 1:} $\#\post(f)\geq 4$.  Let  $X_1, \dots, X_N$ be  a set of 
  $(n+k)$-tiles  whose union is connected and joins opposite sides of
  $\CC$. We may assume that these tiles form a chain joining disjoint
  $0$-edges $E$  and $\widetilde{E}$. To prove the desired inequality,
  we will break this chain into $M$ subchains $X_{s_i}, \dots,
  X_{s_{i+1}-1}$, where $M\in \N$, $i=1, \dots, M$, and $s_1=1<s_2<\dots<s_{M+1}=N+1$. 
    The length of each subchain (i.e., the number 
  $s_{i+1}-s_i$) will be at least $D_n$. The number $M$ of 
  subchains will be at least $D_k$. Thus $N\geq D_n D_k$, and since the minimum over all  $N$ is equal to $D_{n+k}$, the desired inequality follows. 
  
  To guarantee  the desired lower bound on the length, we will ensure  that
   each subchain $X_{s_i}, \dots, X_{s_{i+1}-1}$  
      joins disjoint $k$-cells.
   Then  the length of such a subchain is at least $D_n$ by 
   Lemma~\ref{lem:flowerbds}. 

  To control the number of  subchains, we will associate with each one 
  a $k$-tile $Y_i$. These $k$-tiles $Y_1,\dots, Y_M$ will
  form a $k$-chain joining $E$ and $\widetilde E$, and hence opposite sides of $\CC$.   Thus $M$, which is 
  the number of $k$-tiles in this chain, as well as the number of subchains, is at
  least $D_k$ (by definition of this quantity; see \eqref{def:dk}).

  We now provide the details of the construction, which is illustrated
  in Figure \ref{fig:proofDk}. 
  We will use auxiliary $k$-cells 
  $c_1,c_2, \dots$ of dimension $\le 1$.  
  If $c_i$ is  $0$-dimensional, then $c_i$ consists of a 
 $k$-vertex $p_i$, and we let $W^k(c_i)\coloneqq W^k(p_i)$ (see Definition~\ref{def:flower}). 
If $c_i$ is $1$-dimensional, then $c_i$ is a $k$-edge and    
 $W^k(c_i)$ is the edge flower of
 $c_i$ as in Definition~\ref{def:edgeflower}.

Since $\CC$ is $f$-invariant, the cell decomposition $\DD^{k}$ is a refinement of $\DD^0$.  
Hence there exist disjoint 
$k$-edges $e\sub E$ and $\widetilde e\sub \widetilde E$ with 
$X_1\cap e\ne \emptyset$ and $X_N\cap \widetilde e\ne \emptyset$. 

For some  number $M\in \N$ we will now inductively define  $k$-cells  $c_1, \dots, c_{M+1}$ of dimension $\le 1$, 
$k$-tiles $Y_1, \dots,  Y_M$, and indices $s_1=1< s_2<\dots <s_{M+1}=N+1$  with the following properties: 

\begin{enumerate} 
     
\item
  \label{item:k_tile_chain1}
  $c_1=e$,  $c_{M+1}=\widetilde e$, 
  and $c_{i}\cap c_{i+1}=\emptyset $ for $i=1, \dots, M$. 
      
\item 
  \label{item:k_tile_chain2}
  $c_{i}\cap Y_i\ne \emptyset$ for $i=1, \dots , M$, $c_{i+1}\sub \partial Y_i$
  for $i=1, \dots , M-1$, and $\widetilde e\cap Y_M \ne \emptyset$. 
  
\item 
  \label{item:k_tile_chain3}
  $X_{s_{i}}, \dots,  X_{s_{i+1}-1}$ is an $(n+k)$-chain
  joining $c_{i}$ and $c_{i+1}$ for $i=1, \dots, M$.  
     
\end{enumerate}

     Note that   \ref{item:k_tile_chain1} and \ref{item:k_tile_chain2} imply  that $E\cap Y_1\supset  e\cap Y_1\ne \emptyset$, 
     $\widetilde E\cap Y_M\supset \widetilde e\cap Y_M\ne \emptyset$, and $Y_i\cap Y_{i+1}\supset c_{i+1}\cap Y_{i+1}\ne \emptyset$
     for $i=1, \dots, M-1$. Hence $Y_1, \dots, Y_M$ will be  a $k$-chain joining the  
  $0$-edges $E$ and $\widetilde E$ as desired.

      Let $s_1=1$ and $c_1=e$.
 Suppose first that $\widetilde e$ meets $\overline{W^k(c_1)}$.
   Since $\widetilde e$ is disjoint from $e=c_1$ and hence from 
   $W^k(c_1)$, the points in $\widetilde e\cap 
   \overline{W^k(c_1)}$  lie in  $\partial W^k(c_1)$. By  Lemma~\ref{lem:edgeflower}~\ref{item:edgeflower2} there 
  exists   a $k$-tile $Y_1$ that meets both $c_1$ and 
   $ \widetilde e\supset \widetilde e\cap 
   \overline{W^k(c_1)}$. 
     We let  $M=1$, set $c_2=\widetilde e$,  and stop the construction. We have all the desired properties   \ref{item:k_tile_chain1}--\ref{item:k_tile_chain3}.

 In the other case, where $\widetilde e\cap
 \overline{W^k(c_1)}=\emptyset$,  not all 
the $(n+k)$-tiles  $X_1, \dots, X_N$ are contained in $\overline{W^k(c_1)}$.  So there exists
    a smallest index $s_2\ge  1$ such that $X_{s_2}$ meets 
    $S^2\setminus \overline{W^k(c_1)}$. Then $s_2>s_1=1$, because $X_1$ meets $e=c_1$ and is hence contained in $\overline{W^k(c_1)}$. To see this, we use Lemma~\ref{lem:edgeflower}~\ref{item:edgeflower3} and the fact that $X_1$ is contained in some $k$-tile. Moreover,  for a similar reason we have  $X_{s_2}\sub S^2
    \setminus W^k(c_1)$.  By definition of $s_2$ the set 
    $X_{s_{2}-1}$ is contained in  $\overline{W^k(c_1)}$. Hence
    every point in the non-empty intersection  $X_{s_{2}-1}\cap X_{s_2}$  lies in  $\partial W^k(c_1)$. Note that the $(n+k)$-tiles $X_{s_{2}-1}$ and $ X_{s_2}$ do not necessarily meet in a $k$-vertex; but  by Lemma~\ref{lem:edgeflower}~\ref{item:edgeflower2} there exists a $k$-cell  $c_2\sub \partial W^k(c_1)$ of dimension $\le 1$  that has common points with both $X_{s_{2}-1}$ and $X_{s_2}$,  and a  $k$-tile $Y_1\sub \overline {W^k(c_1)}$ with $ c_1\cap Y_1\ne \emptyset$ and 
   $c_2\sub \partial Y_1$.  Then  $c_1\cap c_2=\emptyset=c_2\cap \widetilde e$, and  the chain $X_{s_1}=X_1, \dots, X_{s_2-1}$ joins $c_1$ and $c_2$. 
    
   We can now repeat the construction as in the first step by using
   the chain $X_{s_2}, \dots, X_N$ that joins the disjoint $k$-cells
   $c_2$ and $\widetilde e$, etc.  If in the process one of the cells
   $c_i$ has dimension $0$, we invoke
   Lemma~\ref{lem:flowerprop}~\ref{item:flower_prop2} and
   \ref{item:flower_prop3} instead of
   Lemma~\ref{lem:edgeflower}~\ref{item:edgeflower2} and
   \ref{item:edgeflower3} in the above construction.  The construction
   eventually stops, and it is clear that we obtain cells and indices
   with the desired properties. The statement follows in this case.

  \smallskip
   {\em Case~2:} $\#\post(f)=3$. Let $E_1,E_2,E_3$ be the three $0$-edges. 
   Consider a connected union $K$ of $(n+k)$-tiles joining  opposite
   sides of $\CC$ with  $N=D_{n+k}$ elements. Then  $K$ meets
   $k$-edges contained in the $0$-edges, say $k$-edges $e_i\sub E_i$
   for $i=1,2,3$. From $K$  we can extract a simple $(n+k)$-chain
   joining $e_1$  and $e_2$ as well as another simple chain that
   joins $e_3$ to one tile $X$ in the chain  joining $e_1$  and
   $e_2$.  Starting from this ``center tile'' $X$, we can find three
   simple $(n+k)$-chains that join $X$ to the edges $e_1,e_2,e_3$,
   respectively, and have only the tile $X$ in common. 
    
   More precisely, for $i=1,2,3$ we can find    
a number
$N_i\in \N_0$ 
 and $(n+k)$-chains 
 $X, X^i_1, \dots, X^i_{N_i}$  that join $X$ and
   $e_i$.  Here the first tile $X$ is the same in all chains and it
   is understood that the chain consists only of  $X$ if $N_i=0$.
   Moreover, all the $(n+k)$-tiles  
   $$X, X^1_1, \dots, X^1_{N_1}, X^2_1, \dots, X^2_{N_2}, X^3_1, \dots, X^3_{N_3}$$
   are pairwise distinct tiles from $K$. Thus their number is bounded 
   by  the number of $(n+k)$-tiles in $K$. Since they still form a
   connected set joining opposite sides of $\CC$, we have 
   $N_1+N_2+N_3+1= N=D_{n+k}$. 
   
  Let $Y$ be the unique $k$-tile with $X\sub Y$, and consider the chain $X,
   X^1_1, \dots , X^1_{N_1}$. 
   Suppose that $Y\cap e_1=\emptyset$. Since $X\sub Y$, we have
   $N_1\ge 1$, and the chain  $ X_1, \dots, X_{N_1}$ joins $Y$ and
   $e_1$. Hence this chain  or a subchain  must also join a  $k$-edge
   $e \sub \partial Y$ and $e_1$.  
   Then $e\cap e_1=\emptyset$. As in the first part of
   the proof, we can find $k$-tiles  $Y_1, \dots , Y_{M_1}$ joining
   $e$ and $e_1$, where $M_1\in \N$  and  $N_1\ge M_1
   D_n$.   
   
   If $Y\cap e_1\neq\emptyset$, we set $M_1=0$ and do not define new
   $k$-tiles. In any case we have that $Y, Y^1_1, \dots, Y^1_{M_1}$
   is a chain joining $Y$ and $e_1$ (again we use the convention that
   this chain consists only of $Y$ if $M_1=0$). We also have  $N_1\ge
   M_1 D_n$ (which is trivial if $M_1=0$).  
   
   Using  a similar construction for the other indices  $i=2,3$,   we
   obtain numbers $M_i\in \N_0$  for each $i=1, 2,3$ that satisfy $N_i\ge
   M_iD_n$,  and chains  
   $Y, Y^i_1, \dots, Y^i_{M_i}$ of $k$-tiles that join $Y$ and $e_i$. 
   The union of these $k$-tiles is a connected set joining opposite sides
   of $\CC$. Therefore, it contains at least $D_k$ distinct elements. On the other hand, the number of distinct $k$-tiles in the union is   
   at most $M_1+M_2+M_3+1$ (note that the
   three chains may have other $k$-tiles in common apart from $Y$). Hence  $D_k\le M_1+M_2+M_3+1$, 
   and it follows that 
   \begin{align*}
   D_n(D_k-1)+1&\le D_n(M_1+M_2+M_3)+1\\ &\le N_1+N_2+N_3+1\le D_{n+k}, 
   \end{align*} 
   which is the desired inequality (\ref{weaksuper}).
   \end{proof}

 \begin{lemma} \label{lem:submultexp}
  Let $f\: S^2\ra S^2$ be a Thurston map with $\#\post(f)\ge 3$,    $\CC\sub S^2$ be a Jordan curve with $\post(f)\sub \CC$, and  $D_n=D_n(f,\CC)$ for $n\in \N_0$. 
 Suppose that  $\CC$ is $f$-invariant.   Then the limit 
$$\Lambda_0\coloneqq\lim_{n\to \infty}D_n^{1/n} $$ 
\index{d0 n@$D_n$}
\index{L0@$\Lambda_0$}
exists 
and $\Lambda_0= \displaystyle \sup_{n\in \N}D_n^{1/n} \le
 \deg(f)$.
  
  If in addition there exists $n_0\in \N$ with $D_{n_0}\ge 2$, then $\Lambda_0>1$  and   $D_n\to \infty$ as $n\to \infty$. 
  \end{lemma}
    The lemma shows that if $f$ is combinatorially expanding for $\CC$,   then 
   $\Lambda_0>1$. So if $\Lambda\in (1,\Lambda_0)$ is arbitrary, then 
 $D_n\gtrsim \Lambda^n$ for all $n$.  
 
 We will see later (Proposition~\ref{prop:exp}) that if $f$ is
 expanding, but $\CC$  is not necessarily $f$-invariant, then the
 limit $\Lambda_0=\lim_{n\to \infty} D_n(f,\CC)^{1/n}$ still
 exists and is independent of $\CC$. Moreover, the improved
 estimate $\Lambda_0\leq \deg(f)^{1/2}$ holds (see
 Proposition~\ref{prop:macgrdn}). 
 
 \begin{proof} We use Lemma~\ref{lem:submult}. 
 First note that inequality \eqref{weaksuper} is also true if $\post(f)=4$. The statements now essentially follow from Fekete's lemma (see \cite[Proposition~9.6.4]{KH}). We provide the 
 details for the convenience of the reader.
 
  A simple induction argument using \eqref{weaksuper} shows that if 
$D_{N}\ge 2$ for some $N\in \N$, then 
\begin{equation} \label{eq:supmult5} 
D_{kN}\ge D_N^{k-1}+1 
\end{equation}  for all $k\in \N$. 
For such $N$ let $k(n)=\lfloor  n/N\rfloor$. Noting that the sequence $\{D_n\}$ is non-decreasing and using \eqref{eq:supmult5},  we obtain 
\begin{align*} 
  \liminf_{n\to \infty} \frac1n \log (D_n)&\ge 
\liminf_{n\to \infty} \frac1n \log (D_{k(n)N})
  \\ &\ge  \liminf_{n\to \infty} \frac{k(n)-1}n \log (D_N)
       = \frac1N \log (D_N). 
\end{align*}
This inequality is trivially true if $D_N=1$, and so
 $$\liminf_{n\to \infty} \frac1n \log (D_n)\ge \sup_{n\in \N}\frac1n \log (D_n). $$
 On the other hand, 
   $$\limsup_{n\to \infty} \frac1n \log (D_n)\le \sup_{n\in \N} 
 \frac1n \log (D_n),  $$
 and so the limit $$\displaystyle \alpha\coloneqq  \lim_{n\to\infty}\frac1n \log (D_n)\quad \text{exists and}\quad
 \alpha=\sup_{n\in \N} \frac1n \log (D_n). $$ 
Note that $D_n\le \#\X^n(f,\CC)\le 2\deg(f)^n$ which implies $\alpha\le \log(\deg(f))$. The first part of the statement now follows by taking exponentials.  

For the last part suppose that  there exists $n_0\in \N$ with $D_{n_0}\ge 2$.   Then
$$\Lambda_0=\sup_{n\in \N}D_n^{1/n} \ge  D_{n_0}^{1/n_0}>1, $$ and it is clear from the definition of $\Lambda_0$ as a limit that $D_n\to \infty$ as $n\to\infty$.
 \end{proof}

 We conclude this section with  a lemma that shows that if a Thurston map is combinatorially expanding, then there are   points that lie ``deep inside'' a given 
  edge or tile. It is related to Lemma~\ref{lem:quasiball}  for expanding Thurston maps.
   
\begin{lemma} \label{lem:vinint}
Let $f\:S^2\ra S^2$ be a Thurston map that satisfies
$\#\post(f)\ge 3$, and let $\CC\sub S^2$ be an $f$-invariant Jordan curve with $\post(f)\sub \CC$. Suppose  that  
  $D_{n_0}(f,\CC)\ge 2$ for some   $n_0\in \N$. 
  
  \index{combinatorially expanding}
  \index{expanding!combinatorially}
  \index{Thurston map!combinatorially expanding}

 \begin{enumerate}
 
  \item
    \label{item:vinint1}
    If $n\in \N_0$ and  $e$ is an $n$-edge, then there exists an $(n+n_0)$-vertex $p$ with $p\in \inte(e)$. 
 
  \item
    \label{item:vinint2}
    If $n\in \N_0$ and $X$ is an $n$-tile, then there exists an
    $(n+n_0)$-edge with $\inte(e)\sub \inte(X)$, and an
    $(n+2n_0)$-vertex $p$ with $p\in \inte(X)$.  
  \end{enumerate} 
  \end{lemma} 

\begin{proof} In the previous statements and the ensuing proof  it  is understood that the term $k$-cell for $k\in \N_0$ refers to a cell in 
$\DD^k=\DD^k(f,\CC)$.

\smallskip
\ref{item:vinint1} 
Suppose $e$ is an $n$-edge that does not contain $(n+n_0)$-vertices in its interior.
By Proposition~\ref{prop:invmarkov}~\ref{item:invmarkov4} the $n$-edge $e$ is equal to the union of all $(n+n_0)$-edges contained in $e$. Thus  $e$ must be an $(n+n_0)$-edge
itself. Let $u$ and $v$ be the endpoints of $e$, and $X$ be an $(n+n_0)$-tile 
containing $e$ in its boundary. Then $K=X$ meets the two disjoint $n$-cells $\{u\}$ and $\{v\}$. Hence by Lemma~\ref{lem:flowerbds} the set $K$ should consist of at least $D_{n_0}=D_{n_0}(f,\CC)\ge 2$ $(n+n_0)$-tiles.   This is a contradiction proving the statement.

\smallskip
\ref{item:vinint2}
Let $X$ be an $n$-tile. By  
Proposition~\ref{prop:invmarkov}~\ref{item:invmarkov3} we know that $X$ is the union of all $(n+n_0)$-tiles contained in $X$. In particular, there exists an $(n+n_0)$-tile $Y$ 
with $Y\sub X$. We claim that there exists an $(n+n_0)$-edge in the boundary of $Y$ that meets $\inte(X)$. Otherwise, $\partial Y\cap \inte(X)=\emptyset$, and as $Y\sub X$, we must have $\partial Y\sub \partial X$. Since both sets $\partial Y$ and $\partial X$ are Jordan curves, this is only possible if $\partial Y=\partial X$. Then $Y$ meets all $n$-vertices contained in $\partial X$, and two distinct $n$-vertices in particular. As in the proof of \ref{item:vinint1}, this leads to a contradiction. 

Hence there exists an $(n+n_0)$-edge $e$ with $e\cap \inte(X)\ne \emptyset$. 
Since $\inte(X)$ is an open subset of $S^2$, we then also have $\inte(e)\cap \inte(X)\ne \emptyset$. 
Since $\DD^{n+n_0}$ is a refinement of $\DD^n$, by Lemma~\ref{lem:mincell} 
we know that there is a unique cell $\tau$ in $\DD^n$ with $\inte(e)\sub \inte(\tau)$. Then
$\inte(\tau)\cap \inte(X)\ne \emptyset$, and so $X=\tau$ by
condition~\ref{item:def_cell2}
in Definition~\ref{def:celldecomp}. 
Hence $\inte(e)\sub \inte(X)$ as desired. 

By \ref{item:vinint1} there exists an $(n+2n_0)$-vertex $p$ with $p\in \inte(e)$. Then we also have $p\in \inte(X)$ as desired. 
\end{proof}

  \section{Two-tile subdivision rules}
\label{sec:subdivisions}

 In this section we consider a $2$-sphere $S^2$ and two given cell decompositions $\DD^0$ and $\DD^1$ of $S^2$. We call  the cells in $\DD^0$ 
 the $0$-cells and the cells in $\DD^1$  the $1$-cells. Similarly, we refer to the tiles in $\DD^0$ as the $0$-tiles, the edges in $\DD^1$ as the $1$-edges, etc.
We know by Proposition~\ref{prop:thurstonex} that under suitable additional assumptions for such a pair $(\DD^1, \DD^0)$ there exists a Thurston  map $f\: S^2\ra S^2$   that is cellular for 
 $(\DD^1, \DD^0)$.  The map $f$ is unique
up to Thurston equivalence if additional data is provided, namely a labeling
 $L\: \DD^1\ra \DD^0$.  The following simple 
 example illustrates  that  the pair  $(\DD^1, \DD^0)$ alone without the labeling is in general not enough to determine the map $f$.

  \begin{figure}
  \centering
  \begin{overpic}
    [width=8cm, tics=10,
    ]{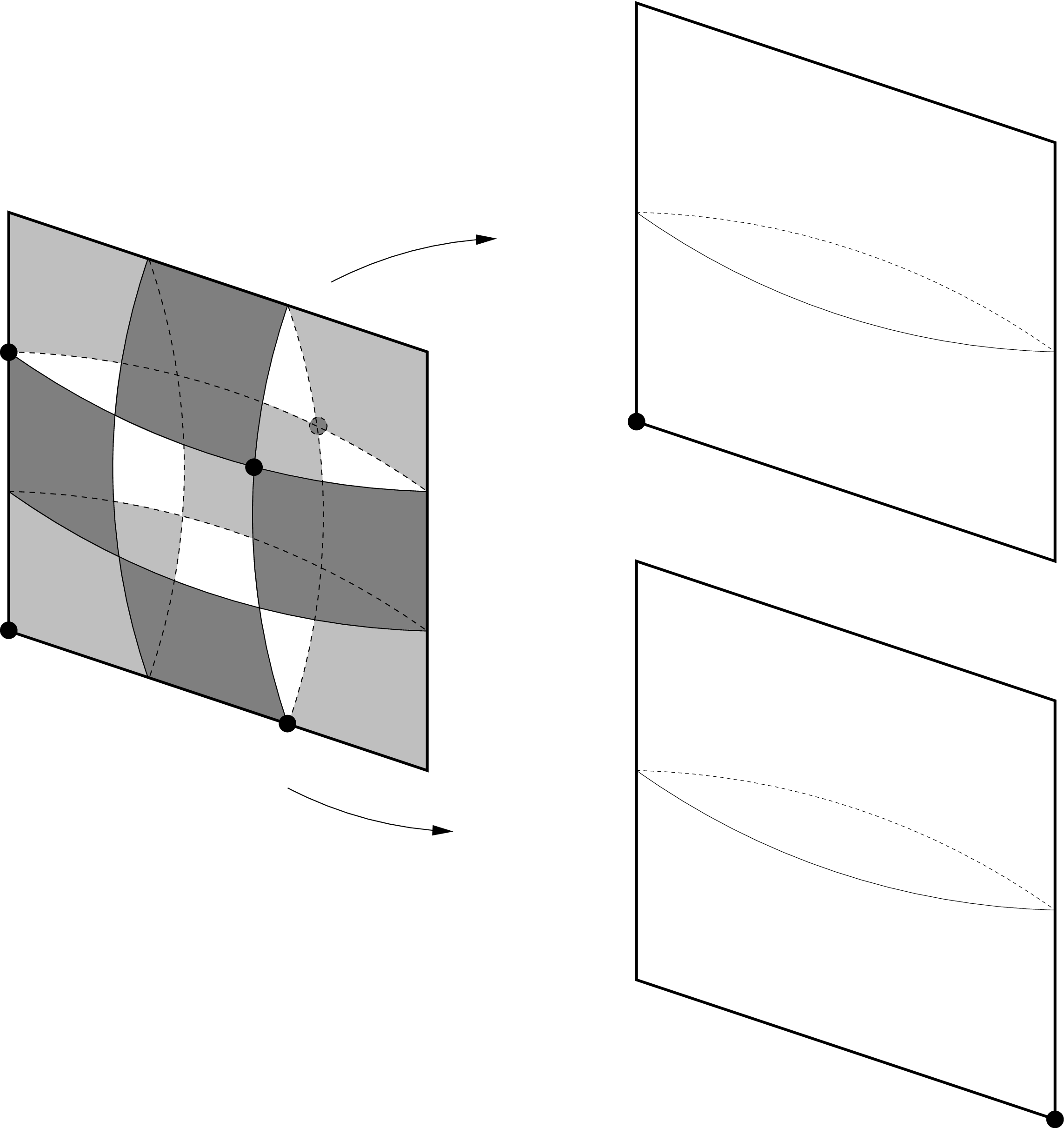}
    \put(32,23){$\scriptstyle{\widetilde{f}}$}
    \put(34,79){$\scriptstyle{f}$}
    \put(55,58.5){$\scriptstyle{0}$}
    \put(93,47){$\scriptstyle 1$}
    \put(93,89){$\scriptstyle 2$}
    \put(55,100.7){$\scriptstyle 3$}
    \put(55,10){$\scriptstyle{3}$}
    \put(93,-3){$\scriptstyle 0$}
    \put(93,39){$\scriptstyle 1$}
    \put(55,51.4){$\scriptstyle 2$}
    \put(-1,40){$\scriptstyle{0}$}
    \put(12,36){$\scriptstyle 1$}
    \put(11.7,51){$\scriptstyle 2$}
    \put(-2,56){$\scriptstyle 3$}
    \put(15,85){$\scriptstyle{\DD^1}$}
    \put(98,70){$\scriptstyle \DD^0$}    
    \put(98,20){$\scriptstyle {\DD}^0$}
  \end{overpic}
  \caption{Two 
    subdivision rules.}
  \label{fig:2_2tilesubdiv}
\end{figure}
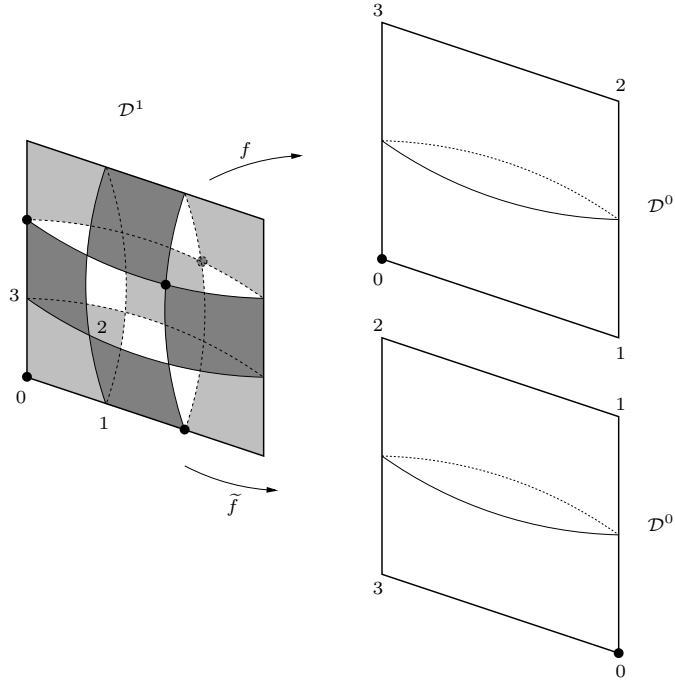

\begin{ex}
  \label{ex:subdiv_diff_labels}
  Figure~\ref{fig:2_2tilesubdiv}
  shows
   two-tile subdivision
  rules $(\DD^1, \DD^0, L)$ and
  $({\DD}^1,{\DD}^0, \widetilde{L})$ on a sphere $S^2$ given as
  a pillow obtained by gluing two copies of a square together
  along their boundaries. The cell decomposition $\DD^0$ is
  shown  twice on the right in  the figure with the two
  squares as the $0$-tiles and their common sides and corners as
  $0$-edges and $0$-vertices, respectively. Each of the two
  $0$-tiles is subdivided into nine squares of equal size. From
  this one obtains a refinement $\DD^1$ of $\DD^0$ as indicated
  on the left in  the figure. We color the tiles black and white
  so that $\DD^1$ and $\DD^0$ are checkerboard tilings. Here the
  top square of the pillow is white and the colors of the
  $1$-tiles are as indicated on the left in 
  Figure~\ref{fig:2_2tilesubdiv}.
    
There are unique orientation-preserving labelings  $L$ and $\widetilde L$ that send $1$-tiles to 
  $0$-tiles of the same colors and the $1$-vertices marked
  as a black dot on the left  to the one $0$-vertex marked in
  the same way in the two representations of $\DD^0$  on the right
  (one can easily see this directly or derive it as a special case of 
   Lemma~\ref{lem:labeluniq}~\ref{item:labeluniq2} below). 
 In the figure we also show  additional markings of some corresponding  vertices for better illustration. 
 
 In this way we get two subdivision rules 
  $(\DD^1, \DD^0, L)$ and $({\DD}^1, {\DD}^0,
  \widetilde{L})$ realized by Thurston maps $f$ and $\widetilde f$ as indicated.
  The maps are uniquely determined if we require in addition that they are piecewise scaling maps on $1$-tiles. 
 The map $\widetilde f$ assigns the same  $0$-tile  to  each  $1$-tile  as $f$, followed by an additional rotation.


  The maps $f$ and $\widetilde{f}$ 
  are not Thurston equivalent. In fact, every
  postcritical point of  $f$ is a fixed point, whereas no postcritical
  point of $\widetilde{f}$ is a fixed point (each postcritical point
  of  $\widetilde{f}$ is periodic with period $4$). Note that both  maps  are (Thurston equivalent to) Latt\`{e}s maps.
\end{ex}

As we know from Section~\ref{sec:tiles},  every Thurston map
$f\:S^2\ra S^2$ arises as a cellular map for a pair of cell
decompositions $\DD^1$ and  $\DD^0$ of $S^2$.  This  gives a
useful description of a Thurston map in combinatorial terms. If
one wants to study the dynamics of $f$, one is interested in the
cell decompositions $\DD^n$ obtained from    pulling back
$\DD^0$ by $f^n$ as in Lemma~\ref{lem:pullback}. In general, in
order to determine the combinatorics of the whole sequence
$\DD^n$, $n\in \N_0$ (i.e., the inclusion and intersection
patterns of cells of possibly distinct 
levels),  is not enough to just  know  the pair $(\DD^1, \DD^0)$ and the labeling $\tau\in \DD^1\mapsto f(\tau)\in \DD^0$, but one also needs specific information on the {\em pointwise}   mapping behavior of $f$ on the cells in $\DD^1$.  Indeed,  suppose   $g$ is another map that is cellular for $(\DD^1, \DD^0)$ and
induces the same labeling as $f$,  i.e., $f(\tau)=g(\tau)$ for all $\tau\in \DD^1$. 
Let $\widetilde \DD^n$ be the cell decomposition of $S^2$  obtained from $\DD^0$ by pulling back by $g^n$. Then one can show (by an argument very similar to the considerations in  the proof of  Lemma~\ref{lem:cexp_Cinv} below) that $\DD^n$ and $\widetilde \DD^n$ are isomorphic cell complexes (see Definition~\ref{def:compiso}) for {\em fixed} $n\in \N_0$.   In contrast,  the intersection 
patterns of corresponding cells $\sigma\in \DD^n$, $\tau\in
\DD^m$ and $\widetilde{\sigma}\in \widetilde \DD^n$,
$\widetilde{\tau}\in \widetilde{\DD}^m$, for different 
levels $n,m\in 
\N_0$,
may not be the same.

The situation changes  if $(\DD^1, \DD^0)$ is a cellular Markov partition  for $f$, because then 
the combinatorics of the sequence $\DD^0,\DD^1,\DD^2, \dots$ is completely determined by 
$(\DD^1, \DD^0)$ and the combinatorial data given by the labeling
$\tau\in \DD^1\mapsto f(\tau)\in \DD^0$ (see
Remark~\ref{rem:two-tile_fV0}~(ii) and Proposition~\ref{prop:subandD^n}).
This suggests that if one wants to study Thurston maps as given by 
Proposition~\ref{prop:thurstonex} from a purely combinatorial point of view, then one should add the additional assumption that $\DD^1$ is a refinement of $\DD^0$. If we restrict 
  ourselves to the case 
where $\DD^0$ contains only two tiles, then we are led to the concept of a 
two-tile subdivision rule as defined in the introduction of this chapter. 

The proofs of Propositions~~\ref{prop:ThmapSub} 
and~\ref{prop:rulemapex} are easy consequences of our previous considerations. 

\begin{proof}[Proof of Proposition~\ref{prop:ThmapSub}] Suppose $f$, $\CC$, $\DD^0$, $\DD^1$, and $L$ are as in the statement. Then $L$ is an orientation-preserving labeling (see Section~\ref{sec:labelings}) and $\DD^1$ is a refinement 
of $\DD^0$  (see  
Proposition~\ref{prop:invmarkov}~\ref{item:invmarkov1}). It now  follows from 
Proposition~\ref{prop:celldecomp} and 
 the discussion after Definition~\ref{def:subdivcomb} that 
$(\DD^0, \DD^1, L)$ is a two-tile subdivision rule realized by $f$.   \end{proof}

 \begin{proof}[Proof of Proposition~\ref{prop:rulemapex}] 
   The first part is just a special case of
   Proposition~\ref{prop:thurstonex}. Note that a map $f$ realizing the given two-tile subdivision rule cannot be a homeomorphism and so must be  a Thurston
   map; indeed, the number of $1$-tiles is equal to $2\deg(f)$,
   and also $>2$ by
   Definition~\ref{def:subdivcomb}~\ref{item:subdivcomb2}. So
   $\deg(f)\ge 2$.
 
   Since $\DD^1$ is a refinement of $\DD^0$, the $1$-skeleton
   $\CC$ of $\DD^0$ is contained in the $1$-skeleton of
   $\DD^1$. Moreover, since $f$ is cellular for $(\DD^1,\DD^0)$,
   this map sends the $1$-skeleton of $\DD^1$ into the
   $1$-skeleton of $\DD^0$. Hence $f(\CC)\sub \CC$, and so $\CC$
   is $f$-invariant.  Each postcritical point of $f$ is a vertex
   of $\DD^0$ and hence contained in $\CC$.
\end{proof}
 
\begin{rem}  
  \label{rem:two-tile_fV0}  
  Let $(\DD^1,\DD^0, L)$ be a two-tile subdivision rule on
  $S^2$, and $f\colon S^2\to S^2$ be a Thurston map that realizes
  $(\DD^1, \DD^0, L)$ according to
  Proposition~\ref{prop:rulemapex}.
 
 \smallskip 
 (i)  If
${\bf V}^0$ denotes the set of vertices of $\DD^0$ and
${\bf V}^1$ the set of vertices of $\DD^1$, 
then we  have
$\crit(f)\sub {\bf V}^1$ and $\post(f)\sub {\bf V}^0$ (see
Lemma~\ref{lem:constrmaps}). In general, $\post(f)\ne  {\bf V}^0$
(see Example~\ref{ex:extra_vertex_D0}).

Since the length of each cycle in
$\DD^0$ is $2$, a vertex $v$ in ${\bf V}^1$ is a critical point
of $f$ if and only if the length of the cycle of $v$ in $\DD^1$
is $\ge 4$ (see Remark~\ref{rem:dd'}).  Hence if $f$ and $g$ both
realize the subdivision rule, then
$\crit(f)=\crit(g)\sub {\bf V}^1$. Moreover, since the orbit of
any point in ${\bf V}^1$ is completely determined by the
labeling, we then also have $\post(f)=\post(g)\sub {\bf V}^0$.
  
Let $\CC$ be the Jordan curve of $\DD^0$, (i.e., the $1$-skeleton of 
$\DD^0$), and  $\DD^n(f,\CC)$ for
  $n\in \N_0$ be the cell decomposition defined according to
  Definition~\ref{def:DDn}. Then $\DD^0 = \DD^0(f,\CC)$
  if and only if 
  $ \V^0=\post(f)$.  If this is not true, then   the  points in
  $\V^0 \setminus \post(f)$ are additional points that are not distinguished 
  by the dynamics  of the
  map $f$. In view of this,  it seems to make sense to include a  requirement that forces  $ \V^0=\post(f)$ in the definition 
  of a two-tile subdivision rule, but we chose not to do so  in order to keep the definition a little simpler.   

\smallskip 
(ii)    Let us  now assume   that $\V^0=\post(f)$. In this case, $\DD^0
  =\DD^0(f,\CC)$, and also  $\DD^1=\DD^1(f,\CC)$ as follows from the uniqueness statement in 
  Lem\-ma~\ref{lem:pullback}. 

 By Proposition~\ref{prop:rulemapex} we know that $\CC$ is
  $f$-invariant and $\post(f) \subset \CC$. This means
  that the cell decompositions $\DD^n\coloneqq \DD^n(f,\CC)$, $n\in \N_0$,  satisfy the properties
  listed in Proposition~\ref{prop:invmarkov}. In particular, 
  $\DD^{n+k}$ is a refinement of $\DD^n$ for  $k,n\in \N_0$. 

 By Proposition~\ref{prop:invmarkov}~\ref{item:invmarkov5} each $1$-cell  $c'$ is subdivided into  $(n+1)$-cells ``in the
  same way''  as the corresponding $0$-cell  $f(c')=L(c')$ is subdivided into $n$-cells. From this   it is intuitively   clear that the ``combinatorics''
of the cells in the sequence $\DD^n$, $n\in \N_0$, that is, their inclusion and intersection patterns,  can be determined inductively from  
the subdivision rule $(\DD^1, \DD^0, L)$  independently of the map realizing it. 

To  make this more precise, suppose $g\: S^2\ra S^2$ is another
Thurston map realizing the subdivision rule. We denote by
$\DD^\infty$ the disjoint union of the  cell decompositions
$\DD^n$, $n\in \N_0$, i.e., 
$\DD^\infty$ is the set of all cells for $(f,\CC)$, 
where cells with the same underlying sets, but of different levels, are considered distinct. Similarly, let $\widetilde \DD^\infty$ be the set of all cells for $(g,\CC)$. Then there exists a bijection 
$\varphi \:\DD^\infty \ra  \widetilde \DD^\infty$ that preserves the level and the dimension  of each cell, and all inclusion patterns (i.e., $\sigma\sub \tau$ for
 $\sigma, \tau \in \DD^\infty$ if and only if $\varphi(\sigma)\sub \varphi(\tau)$).  This  last property of $\varphi$ implies that also all intersection patterns are preserved (i.e., $\sigma\cap \tau\ne \emptyset $ for
 $\sigma, \tau \in \DD^\infty$ if and only if $\varphi(\sigma)\cap \varphi(\sigma)
 \ne \emptyset$).  In this sense, $\DD^\infty$ and  $\widetilde \DD^\infty$
 have  exactly the same 
 combinatorics.
 
 This will not be proved here, but we refer to
 Proposition~\ref{prop:subandD^n} for a related
 statement. \end{rem}

We will later need a more general version of the uniqueness part
of Proposition~\ref{prop:rulemapex}. To formulate this, we first
introduce a suitable notion of an {\em isomorphism} between
two-tile subdivision rules.

\index{two-tile subdivision rule!isomorphism of}
\index{subdivision!isomorphic}
\index{isomorphism!of two-tile subdivision rules}

Let $(\DD^1,\DD^0,L)$ and  $(\widetilde \DD^1,\widetilde \DD^0,\widetilde L)$ 
be two-tile subdivision rules on $2$-spheres $S^2$ and $\widetilde S^2$, respectively. We say that these subdivision rules are {\em isomorphic} 
if there exist cell complex isomorphisms $\phi_i\: \DD^i\ra \widetilde \DD^i$ for $i=0,1$ such that $\widetilde L(\phi_1(\tau))=\phi_0(L(\tau))$ for $\tau\in \DD^1$ and such that $\sigma\sub \tau$ for $\sigma\in \DD^1$, $\tau \in\DD^0$ if and only if $\phi_1(\sigma)\sub \phi_0(\tau)$.

If, by abuse of notation, we denote the image of a cell $\tau\in \DD^i$ under $\phi_i$  by $\widetilde \tau$ for $i=0,1$, then the last two conditions require that $\widetilde L(\widetilde \tau)=\widetilde {L(\tau)}$ for $\tau\in \DD^1$ and that $\sigma\sub \tau$ for $\sigma\in \DD^1$, $\tau \in\DD^0$ if and only if 
$\widetilde \sigma\sub \widetilde\tau$. So $(\DD^1,\DD^0,L)$ and  $(\widetilde \DD^1,\widetilde \DD^0,\widetilde L)$ are isomorphic if the combinatorics of cells under the correspondence 
$\tau \leftrightarrow \widetilde \tau$ is the same and if this correspondence is also compatible with the labelings.

\begin{lemma}
  \label{lem:isotwotequiv} 
  \index{Thurston!equivalent}
  Let $(\DD^1,\DD^0,L)$ and  $(\widetilde \DD^1,\widetilde \DD^0,\widetilde L)$ 
  be isomorphic two-tile subdivision rules on $2$-spheres $S^2$ and $\widetilde
  S^2$, respectively. Suppose the Thurston map $f\: S^2\ra S^2$
  realizes $(\DD^1,\DD^0,L)$ and the Thurston map $\widetilde f\:
  \widetilde S^2\ra \widetilde S^2$ realizes $(\widetilde
  \DD^1,\widetilde \DD^0,\widetilde L)$. Then $f$ and $\widetilde f$
  are Thurston equi\-valent.
\end{lemma} 

\begin{proof} 
  The proof uses very similar ideas as the proof of the
  uniqueness part of Proposition~\ref{prop:thurstonex}.

  By assumption, there exist cell complex isomorphisms
  $\phi_i\: \DD^i\ra \widetilde \DD^i$ for $i=0,1$ with the 
  following properties: $\widetilde L(\widetilde \tau)=\widetilde {L(\tau)}$ 
  for each  $\tau\in \DD^1$,  and  for
  $\sigma\in \DD^1$, $\tau \in\DD^0$ we have the inclusion
  $\sigma\sub \tau$ if and only if
  $\widetilde \sigma\sub \widetilde\tau$. Here we denote the
  image of a cell $\tau\in \DD^i$ under $\phi_i$ by
  $\widetilde \tau$ for $i=0,1$.

By Lemma~\ref{lem:isocellhomeo}~\ref{item:isocellhomeo2} there exists a homeomorphism $h_0\: S^2\ra \widetilde S^2$ such that $h_0(\tau)=\widetilde \tau$ for each $\tau \in \DD^0$.

The map $f$ realizes the subdivision rule $(\DD^1,\DD^0,L)$. So if  $\tau\in \DD^1$, then 
$f(\tau)=L(\tau)\in \DD^0$. Since $\widetilde f$ realizes $(\widetilde \DD^1,\widetilde \DD^0, \widetilde L)$, 
we have 
$$ \widetilde f( \widetilde \tau)=  \widetilde L( \widetilde \tau)= \widetilde{L(\tau)}= \widetilde{f(\tau)}=
h_0(f(\tau)) \in  \widetilde \DD^0$$
for $\tau\in \DD^1$. The map $f$ is cellular for $(\DD^1, \DD^0)$ and so 
for each $\tau \in \DD^1$ the map
 $f|\tau$ is a homeomorphism of $\tau$ onto $f(\tau)$.  Similarly,  the map  
$\widetilde  f|\widetilde \tau$ is a homeomorphism of $\widetilde \tau$  onto $ \widetilde f( \widetilde \tau)={h_0(f(\tau))}$. Hence for each $\tau\in \DD^1$ the map
$$\varphi_{\tau}\coloneqq (\widetilde f|\widetilde \tau)^{-1}\circ h_0 \circ (f|\tau)$$ is well-defined and a homeomorphism from $\tau$ onto $\widetilde \tau$.
If $x\in \tau$, then $y=\varphi_\tau(x)$ is the unique point $y\in \widetilde \tau $ with $\widetilde f(y)=h_0(f(x))$. 

If $\sigma, \tau \in \DD^1$ and 
$\sigma\sub \tau$, then 
$$ \varphi_\tau|\sigma=\varphi_\sigma. $$
Indeed, if $x\in \sigma$, then $y=\varphi_\sigma(x)
\in \widetilde \sigma\sub \widetilde\tau$ and 
$\widetilde f(y)=h_0(f(x))$. Hence
$\varphi_\sigma(x)=y=\varphi_\tau(x)$ by the uniqueness property of $\varphi_\tau(x)$. 

If a point $x\in S^2$ lies in two cells $\tau,\tau'\in \DD^1$, then
$\varphi_\tau(x)=\varphi_{\tau'}(x)$. Indeed, there exists a unique
cell $\sigma\in \DD^1$ with $x\in \inte(\sigma)$. Then $\sigma\sub \tau
\cap \tau'$ by Lemma~\ref{lem:celldecompint}~\ref{item:cell_decomp2},
and so, by what we have just seen, we conclude
$$ \varphi_\tau(x)=\varphi_\sigma(x)= \varphi_{\tau'}(x). $$

This allows us to define a map $h_1\: S^2\ra \widetilde S^2$ as follows. If $x\in S^2$, pick a  cell $\tau\in \DD^1$ with $x\in \tau$, and set 
$$h_1(x)=\varphi_\tau(x).$$ Then $h_1$ is well-defined. 

The  definitions of $h_1$ and $\varphi_\tau$ imply that $h_0\circ f=\widetilde f\circ h_1$. 
Moreover,  $h_1|\tau=\varphi_\tau$ is a homeomorphism of $\tau$ onto 
$\widetilde \tau$ for each $\tau \in  \DD^1$. So by 
Lemma~\ref{lem:isocellhomeo}~\ref{item:isocellhomeo1}  the map $h_1$ is a 
homeomorphism of $S^2$ onto $\widetilde S^2$.  

We claim that $h_1(\tau)=\widetilde \tau$ not only for $\tau\in \DD^1$, but also for 
$\tau\in \DD^0$. Indeed, let $\tau\in \DD^0$ and $x\in \tau$  be arbitrary. Since 
$\DD^1$ refines $\DD^0$, there exists $\sigma \in \DD^1$ such that $x\in \sigma\sub \tau$. 
Then $\widetilde \sigma\sub \widetilde \tau$ and $h_1(x)\in h_1(\sigma)=\widetilde \sigma \sub \widetilde \tau$. We conclude that $h_1(\tau)\sub \widetilde \tau$.  Conversely, if $y\in \widetilde 
\tau$,  then there exists a cell in $\widetilde \DD^1$ that contains $y$ and is contained in 
$\widetilde \tau$. This cell has the form $\widetilde \sigma$ with $\sigma\sub \tau$. Hence $y\in \widetilde \sigma=h_1(\sigma)\sub h_1(\tau)$. This shows that $h_1(\tau)=\widetilde \tau$
as desired. 

It follows that the homeomorphism $h_1^{-1}\circ h_0\: S^2\ra S^2$ satisfies
$(h_1^{-1}\circ h_0)(\tau)=h_1^{-1}(\widetilde \tau) =\tau$ for each $\tau \in \DD^0$. 
So by Lemma~\ref{lem:isocellhomeo}~\ref{item:isocellhomeo3} the homeomorphism $h_1^{-1}\circ h_0$ is isotopic to $\id_{S^2}$ rel.\ ${\bf V}^0$, where ${\bf V}^0$ is the set of vertices of $\DD^0$. If we postcompose an  isotopy rel.\ ${\bf V}^0$ between $h_1^{-1}\circ h_0$  and $\id_{S^2}$ with 
 $h_1$, we see  that $h_0$ and $h_1$ are isotopic rel.\ ${\bf V}^0$, and hence also isotopic rel.\ $\post(f)$, because $\post(f)\sub {\bf V}^0$. Since $h_0\circ f=\widetilde f\circ h_1$, the maps $f$ and $\widetilde f$ are Thurston equivalent. 
\end{proof} 

\begin{rem} 
  \label{rem:isomsub} 
 Suppose the setup is as in the previous lemma and its proof. We
 record some observations that will be important later in the proof of Theorem~\ref{thm:rat_finitelyC}. 

\smallskip
(i) Let $\CC$ be the Jordan curve of 
$\DD^0$ and $\widetilde \CC$ be the Jordan curve of $\widetilde \DD^0$. Then $\CC$ and $\widetilde \CC$ are the $1$-skeletons of $\DD^0$ and $\widetilde \DD^0$, respectively. The homeomorphisms 
$h_0$ and $h_1$ constructed in the previous proof have the property that 
they  send each cell $\tau\in \DD^0$ to the corresponding cell $\widetilde \tau \in  \widetilde \DD^0$. Since the bijection $\tau \mapsto \widetilde \tau$ is an isomorphism of the cell complexes 
$\DD^0$ and $\widetilde \DD^0$,  it preserves the dimension of a cell. This implies that 
the maps $h_0$ and $h_1$ send the $1$-skeleton of $\DD^0$ to the $1$-skeleton of $\widetilde \DD^0$, and so $h_0(\CC)=\widetilde \CC=h_1(\CC)$.  

\smallskip
(ii) The cell complex isomorphisms $\phi_i\: \DD^i\ra \widetilde \DD^i$, $i=0,1$, as in the definition of an isomorphism between $(\DD^1,\DD^0,L)$ and  $(\widetilde \DD^1,\widetilde \DD^0,\widetilde L)$  send flags in $\DD^i$ to flags
in $\widetilde \DD^i$. In  the definition of an  isomorphism of two-tile subdivision rules one can make  the stronger  additional requirement that positively-oriented flags are sent to positively-oriented flags. If one assumes this stronger notion of isomorphism in the previous lemma, then the maps 
$h_0$ and $h_1$ will be orientation-preserving. 

\smallskip (iii)  For a fixed number of cells in $\DD^1$, there is
only a finite number of two-tile subdivision rules
$(\DD^1, \DD^0,L)$ up to isomorphism. This implies that  if we consider  Thurston
maps $f\colon S^2\to S^2$ of fixed  degree and a fixed number of
postcritical points, then up to isomorphism there is only a finite number of
two-tile subdivision rules given by such a map $f$ and an
$f$-invariant Jordan curve $\CC\subset S^2$ according to
Proposition~\ref{prop:ThmapSub}. These statements remain true 
if one uses the strong notion of isomorphism 
between  two-tile subdivision rules  as in (ii).
 
Based on these observations one can show 
 that a rational expanding Thurston map $f\: \CDach \ra \CDach$ with hyperbolic
orbifold has at most finitely many $f$-invariant Jordan curves
$\CC\subset \Cdach$ with $\post(f) \subset \CC$ (see
Theorem~\ref{thm:rat_finitelyC}).

\end{rem} 

If one wants to discuss specific examples of Thurston maps that realize a given two-tile subdivision rule $(\DD^1, \DD^0, L)$, then it is convenient to represent the relevant data
in a compressed form. The  information on the labeling  $L$ is completely determined by a pair of corresponding  positively-oriented flags in 
$\DD^1$ and $\DD^0$.

\begin{lemma} 
  \label{lem:labeluniq}
  \index{labeling!subdivisions}
  Let $(\DD^1, \DD^0)$ be a pair of cell decompositions of  $S^2$
  satisfying conditions
  \ref{item:subdivcomb1}--\ref{item:subdivcomb4} in 
  Definition~\ref{def:subdivcomb}.
  \begin{enumerate}
    
\item
  \label{item:labeluniq1}
Let $(c'_0,c'_1,c'_2)$ and  $(c_0,c_1,c_2)$  be 
     positively-oriented flags 
     in $\DD^1$ and $\DD^0$, respectively. 
     Then there 
     exists 
     a unique  orientation-preserving labeling $L\:\DD^1\ra \DD^0$ with $(L(c'_0), L(c'_1), L(c'_2))=(c_0,c_1,c_2)$.

\item
  \label{item:labeluniq2}
 Let   $v'$ be a vertex and $X'$ be a  tile in $\DD^1$, and $v$ be a vertex and $X$ be a tile in $\DD^0$. If $v'\in X'$, then there exists a unique orientation-preserving labeling $L\:\DD^1\ra \DD^0$ such that $L(v')=v$ and $L(X')=X$. 
\end{enumerate}
\end{lemma}
  
 So in both cases, $(\DD^1, \DD^0, L)$ is a two-tile subdivision rule. In \ref{item:labeluniq2}
 we automatically have $v\in X$, because every vertex in  $\DD^0$ is contained in each of the two 
 tiles in $\DD^0$. 
Note that  $L(v')$ in \ref{item:labeluniq2} is defined,
since for labelings we do not distinguish between $0$-dimensional cells and vertices of a cell decomposition.
 
  Statement  \ref{item:labeluniq1} easily follows from  \ref{item:labeluniq2}.
We formulated \ref{item:labeluniq1} explicitly, because this version  puts the statement in a more conceptual setting and because in the proof we  use 
Lemma~\ref{lem:labelexis} to first establish 
 \ref{item:labeluniq1}   and then derive 
 \ref{item:labeluniq2}.
      
 \begin{proof} For $i=0,1$    denote by ${\bf V}^i, \E^i, \X^i $  the set of vertices, edges, and tiles of $\DD^i$, respectively. Every tile in $\DD^0$  or $\DD^1$ is a $k$-gon for fixed  $k\ge 3$.  
 
 \smallskip 
\ref{item:labeluniq1} To describe the labeling  for  $(\DD^1, \DD^0)$,  we proceed 
in the manner discussed  after Definition~\ref{def:labeldecomp} and choose a particular index set $\mathcal{L}$ for the labeling of the elements 
in $\DD^0$ and  $\DD^1$.

 We let $\mathcal{L}$ be the set that consists of two disjoint copies of $\Z_k$ (one will be for the vertices, and one for the edges), and the set $\{\tt b, \tt w\}$, where again we think of $\tt w$ representing ``white'' and $\tt b$ representing  ``black''. 
 
 We assign to $c_2\in \X^0$ the color ``white'', and ``black'' to the other tile in $\X^0$. We assign $0\in \Z_k$  to the $0$-vertex 
 $v_0\in c_0$. Then there is a unique way to assign labels in $\Z_k$ to the other vertices on $\CC\coloneqq \partial c_2$ (and the corresponding cells of dimension $0$) such that if $v_0, v_1, \dots, v_{k-1}$ are the vertices  indexed by their label, then they are in cyclic order on $\CC$ as considered as the boundary of the white $0$-tile and in anti-cyclic order for  the black $0$-tile.  
  Each $0$-edge $e$ 
is an arc on $\CC$ with endpoints $v_l$ and $v_{l+1}$ for a unique $l\in\Z_k$. We  label $e$ by $l$ (where we think of $l$ as belonging to 
the second copy of $\Z_k$).  Since $(c_0, c_1, c_2)$ is a positively-oriented flag, and $v_0$ is the initial point of $c_1$,  the edge $c_1$ has the label $0$.  All this is just a special case of  Lemma~\ref{lem:labelexis}.
If in this way we assign  to each element in $\DD^0$  a label in $\mathcal{L}$,  we get a bijection  $\psi\: \DD^0\ra \mathcal{L}$.   Note that if $(\tau_0, \tau_1,\tau_2)$ is any positively-oriented flag in $\DD^0$, then its image under $\psi$ 
has the form $(l,l,\tt w)$ or $(l,l-1,\tt b)$ for some $l\in \Z_k$ 
(see Lemma~\ref{lem:labelexis}~\ref{item:labelexis6}). 

For $\DD^1$ we invoke Lemma~\ref{lem:labelexis} directly to set up a suitable map 
$\varphi\:\DD^1\ra \mathcal{L}$. Since $\DD^1$ satisfies the conditions of Lemma~\ref{lem:labelexis},  we can find maps $L_{\bf V}\:{\bf V}^1\ra \Z_k$, $L_\E\: \E^1\ra \Z_k$,  and 
$L_\X\: \X^1\ra\{ \tt b, \tt w\}$ with the properties \ref{item:labelexis2}--\ref{item:labelexis4} stated in the lemma and the normalizations $L_{\bf V}(c'_0)=0$, $L_\E(c'_1)=0$, and $L_\X(c'_2)=\tt w$. 
 The maps $L_{\bf V}$, $L_\E$, $L_\X$ induce a unique  map $\varphi\:  \DD^1\to \mathcal{L}$ such that $\varphi(c)=L_\X(c)$ if $c$ is a $1$-tile,  $\varphi(c)=L_\E(c)$ if $c$ is a $1$-edge, and 
   $\varphi(c)=L_{\bf V}(v)$ if $c=\{v\}$ consists of a $1$-vertex $v$.  Here it is understood that  edges and vertices  in $\DD^1$ map to different copies of $\Z_k$ in  $\mathcal{L}$.

Now define  $L\coloneqq \psi^{-1}\circ \varphi:\ \DD^1\ra \DD^0$. The map $L$ assigns to each  $1$-cell $c$ 
the unique $0$-cell that has the same dimension as $c$  and  carries the same label in $\mathcal{L}$ as $c$.
 
It follows immediately from the properties of the maps 
$\psi$ and $\varphi$ that $L$ preserves dimensions, respects inclusions, and is injective on cells. Hence $L$ is a labeling 
 according to Definition~\ref{def:labeldecomp}. 
 By our normalizations  the map  $L$ sends  the flag $(c'_0,c'_1,c'_2)$ to 
  $(c_0,c_1,c_2)$. 
  
Moreover,  $L$ is
   orientation-preserving. Indeed, $\varphi$ maps the cells $\tau_0, \tau_1, \tau_2$ in a positively-oriented flag in $\DD^1$ to  
  $l$, $l$, $\tt w$, or to  $l$, $l-1$, $\tt b$,  respectively, where $l\in \Z_k$.  These triples correspond to 
  positively-oriented flags in 
  $\DD^0$. 
  It follows that $L$ has the desired properties. 

  To show uniqueness, we reverse the process. Given a labeling
  $L\colon \DD^1\to \DD^0$ with  the stated properties, we use the same 
  map $\psi\: \DD^0\ra \mathcal{L}$ as above and define maps 
  $L_{\bf V}\: {\bf V}^1\ra \Z_k$, $L_\E\:\E^1\ra \Z_k$, $L_\X\: \X^1\ra 
  \{\tt b, \tt w\}$ such that $L_\X(c)=(\psi\circ L)(c)$ if $c$ is a $1$-tile,  $L_\E(c)=(\psi\circ L)(c)$ if $c$ is a $1$-edge, and 
   $L_{\bf V}(v)=(\psi\circ L)(c)$ if $c=\{v\}$ consists of a $1$-vertex $v$.
   
 Then we have   normalizations $L_{\bf V}(c'_0)=0$, $L_\E(c'_1)=0$,  and  
  $L_\X(c'_2)=\tt w$ as in Lemma~\ref{lem:labelexis}~\ref{item:labelexis1}. If we can show that $L_{\bf V}$, $L_\E$,
  $L_\X$ have the properties \ref{item:labelexis2}--\ref{item:labelexis4} in Lemma~\ref{lem:labelexis}, then the uniqueness of $L$ will follow from the corresponding uniqueness statement in this lemma.
    
  To see this, let $e\in \DD^1$ be arbitrary, and $X,Y\in \DD^1$ be the two tiles that contain $e$ in its  boundary. Let $u,v\in {\bf V}^1$ be the two endpoints of $e$. We may assume that notation is chosen so that the flag $(\{u\},e,X)$ is 
  positively-oriented. Then 
  $(\{v\},e,Y)$ is also positively-oriented. It follows that 
  the images of these flags  under $L$ are 
  positively-oriented. Since $L$ is injective on cells, and so $L(u)\ne L(v)$, this   implies that 
  $L(X)\ne L(Y)$. So $L(X)$ and $L(Y)$ carry different colors (given by  $\psi$) which implies that $X$ and $Y$ also carry 
  different colors by  definition of $L_\X$. Hence $L_\X$ has
  property \ref{item:labelexis2} in Lemma~\ref{lem:labelexis}. By
  switching the notation 
for $u$ and $v$ as well as $X$ and $Y$ 
if necessary, we may assume that $X$ and $L(X)$ are white tiles. Since the flag $(\{L(u)\}, L(e), L(X))$ 
  is positively-oriented, and $L(X)$ is white, it follows that  for some $l\in \Z_k$ we have $\psi(L(u))=l$ and $\psi(L(e))=l$. Hence $L_{\bf V}(u)=l$ and $L_\E(e)=l$. Similarly, using that $L(Y)$ is black and that $(\{L(v)\}, L(e), L(Y))$ is positively-oriented, we see that  $L_{\bf V}(v)=l+1$. 
  
  In other words, if we run along an oriented edge $e$ in $\DD^1$ so that a white tile lies on the left  of $e$, then the label  of the endpoint of $e$ (given by $L_{\bf V}$) is increased by one, and decreased by one if a black tile lies on the left. Hence $L_{\bf V}$ has the property \ref{item:labelexis3} in Lemma~\ref{lem:labelexis}.  
  Moreover, we also see that the label $L_\E(e)$ is related  to  the labels of its endpoints as in statement \ref{item:labelexis4} of  Lemma~\ref{lem:labelexis}. The uniqueness of $L$ follows. 
  
  \smallskip 
\ref{item:labeluniq2} If  $v'\in {\bf V}^1$ and $v'\in X'$, then  we have $v'\in \partial X'$. There are precisely two edges in ${\bf E}^1$ that are contained in $\partial X'$ and have $v'$ as one of their endpoints. For precisely one of these edges $e'$, the triple $(\{v'\}, e', X')$ is the unique  positively-oriented flag in $\DD'$ that includes $\{v'\}$ and $X'$.  
 
  Similarly, there exists a unique 
edge $e\in {\bf E}^0$ 
such that 
     $(\{v\}, e, X)$ is a positively-oriented flag in
     $\DD^0$. Thus by \ref{item:labeluniq1}, there exists an   orientation-preserving labeling $L\: \DD^1 \ra \DD^0$ that sends $(\{v'\}, e', X')$ to $(\{v\}, e, X)$. In particular, $L(v')=v$ and $L(X')=X$. This shows existence of a labeling as desired. 
     
To prove uniqueness, suppose that 
 $L\: \DD^1 \ra \DD^0$ is an orientation-preserving labeling with 
 $L(v')=v$ and $L(X')=X$. Then the image of $(\{v'\}, e', X')$ under $L$ 
 is a positively-oriented flag of the form $(\{v\}, L(e'), X)$. Since 
 $(\{v\}, e, X)$ is the unique  positively-oriented flag in $\DD^0$ that includes $\{v\}$ and $X$, we have $L(e')=e$. Uniqueness of $L$ now follows from \ref{item:labeluniq1}. 
      \end{proof}

 If we are given cell decompositions $\DD^1$ and $\DD^0$ as in the last lemma, then by part \ref{item:labeluniq2} we can specify a unique labeling so that $(\DD^1,\DD^0, L)$ becomes a  two-tile subdivision rule  in a very condensed form: all we need to know is the image $0$-tile 
 $X$ of one $1$-tile $X'$, and the image $0$-vertex $v\in X$ of one $1$-vertex $v'\in X'$. In specific examples (see Section~\ref{sec:examples-two-tile}), one usually wants to include more information on the labeling to get a better understanding of the mapping properties of the Thurston map that realizes the subdivision rule.

 Let $f$ be a map realizing  a two-tile subdivision rule $(\DD^1,\DD^0,L)$.   We want to show next that  the property of  $f$ being combinatorially expanding for the Jordan curve $\CC$ of $\DD^0$ 
    is independent of the realization. In contrast,  this   is not true for 
 expansion of the map  (see Example~\ref{ex:barycentric}). 
 We require a lemma.

 \begin{lemma} 
   \label{lem:cexp_Cinv}
   \index{combinatorially expanding}
   \index{expanding!combinatorially} 
   \index{Thurston map!combinatorially expanding}

   Let $f\: S^2\to S^2$ and
   $g\:\widetilde S^2\ra \widetilde S^2$ be Thurston maps. 
   Suppose
   $\# \post(f)\ge 3$, $\CC\subset S^2$ is an $f$-invariant
   Jordan curve with $\post(f)\subset \CC$, and 
   $h_0,h_1\: S^2 \ra \widetilde S^2$ are orientation-preserving
   homeomorphisms satisfying $h_0|\post(f)=h_1|\post(f)$,
   $h_0\circ f=g\circ h_1$, and $h_0(\CC)=h_1(\CC)$.
  
   Then $f$ is combinatorially expanding for $\CC$ if and only if
   $g$ is combinatorially expanding for
   $\widetilde \CC\coloneqq h_0(\CC)=h_1(\CC)$.
 \end{lemma}

As we will see momentarily, the conditions of the lemma imply that $\widetilde{\CC}$
is $g$-invariant. 
 
 \begin{proof} We have  $\post(g)=h_0(\post(f))=h_1(\post(f))$
 (see the remark after Lem\-ma~\ref{lem:T-eq_crit_post}).  Hence $\#\post(g)=\#\post(f)\ge 3$. Moreover, $\widetilde \CC\sub \widetilde S^2$ is a Jordan curve 
 with $\post(g)\sub \widetilde \CC$. This curve is $g$-invariant, since 
 $$g(\widetilde \CC)=g(h_1(\CC))=h_0(f(\CC))\sub h_0(\CC)=\widetilde \CC. $$
So the statement that $g$ is combinatorially expanding for  
$\widetilde \CC$ is meaningful (see Defi\-nition~\ref{def:combexp}). 

Pick  an orientation of $\CC$. By our assumptions the map $\varphi\coloneqq h_1^{-1}\circ h_0$ fixes the elements of $\post(f)$ pointwise and the Jordan curve $\CC$ setwise. Since $\#\post(f)\ge 3$ and 
$\post(f)\sub \CC$,  this implies that $\varphi$ preserves the orientation of $\CC$. Since $\varphi$ is an orientation-preserving homeomorphism on $S^2$, the map   $\varphi$ sends each of the 
complementary components of $\CC$ to itself. Thus, $\varphi$ 
is cellular for $(\DD^0, \DD^0)$, where $\DD^0=\DD^0(f,\CC)$, and we have $\varphi(c)=c$ for each cell $c\in \DD^0$. As in the proof of Lemma~\ref{lem:isocellhomeo}~\ref{item:isocellhomeo3}, this implies that $\varphi$ is isotopic to 
$\id_{S^2}$ rel.\ $\post(f)$.  Hence  $h_0=h_1\circ \varphi$ is isotopic 
to $h_1=h_1\circ \id_{S^2}$ rel.\ $\post(f)$, and so there exists an isotopy $H^0\: S^2\times I\ra \widetilde S^2$ rel.\ $\post(f)$ with 
$H^0_0=h_0$ and $H^0_1=h_1$.

As in the proof of Theorem~\ref{thm:exppromequiv},  we can  repeatedly   lift  the initial isotopy  $H^0$ based on Proposition
  \ref{prop:isotoplift}. In this way  we can find isotopies $H^n\: S^2\times I\ra \widetilde S^2$  rel.\ $\post(f)$  such that 
$H^n_t\circ f=g\circ H^{n+1}_t$ and $H^{n+1}_0=H^n_1$ for all $n\in \N_0$ and $t\in I$. Note that $H^n$ for $n\ge 1$ is actually an isotopy rel.\ $f^{-1}(\post(f))\supset \post(f)$.

Define homeomorphisms  $h_n\coloneqq  H^n_0$ for 
  $n\in\N_0$ (note that  for $n=0$ and $n=1$ these maps agree with our given maps $h_0$ and $h_1$). Then 
  $h_n\circ f=g\circ h_{n+1}$, and so 
  \begin{equation} \label{eq:fngn}
  h_0\circ f^n=g^n\circ h_n
  \end{equation}  for all $n\in \N_0$.
  
We have $h_n|\post(f)=h_0|\post (f)$ which implies  
 \begin{equation}\label{eq:hnsame}
 h_n(\post(f))=\post(g)
  \end{equation}
  for all $n\in \N_0$. Moreover,
  $h_n|f^{-1}(\post(f))=h_1|f^{-1}(\post(f))$ and so 
 \begin{equation} \label{eq:hnfgpre}
  h_n(f^{-1}(\post(f)))=g^{-1}(\post(g))
  \end{equation}  for
  $n\in \N$  as follows from 
  Lemma~\ref{lem:lifts_inverses}.

Our hypotheses  imply that    if $c$ is a cell in $\DD^0(f,\CC)$, then $h_0(c)$ is a cell in 
  $\DD^0(g,\widetilde \CC)$.  Since the set 
  $$\widetilde \DD^n\coloneqq  \{ h_n(c): c\in \DD^n(f,\CC)\}$$ is a cell decomposition of $\widetilde S^2$, it follows from this and   \eqref{eq:fngn}  
  that $g^n$ is cellular for $(\widetilde \DD^n, \DD^0(g, \widetilde \CC))$.  Since $g^n$ is also cellular for the pair $(\DD^n(g,\widetilde \CC), \DD^0(g, \widetilde \CC))$,   the uniqueness statement in Lemma~
  \ref{lem:pullback} implies that $\widetilde \DD^n=\DD^n(g,\CC)$ for all 
  $n\in \N_0$.  In other words, the $n$-cells for $(g,\widetilde \CC)$ are precisely the images of the $n$-cells for $(f,\CC)$ under the homeomorphism $h_n$.

  We also have 
\begin{equation}\label{eq:varccc}
h_n(\CC)=\widetilde \CC 
\end{equation} 
for each $n\in \N_0$. This can be seen by induction on $n$ as follows. 
The statement is true for $n=0$ and $n=1$ by our assumptions and by the definition of $\widetilde \CC$. Assume that $h_n(\CC)=\widetilde \CC$ for some $n\in \N$. Then by   
 Lemma \ref{lem:lifts_inverses} and the induction hypotheses  we have 
  $$J\coloneqq h_{n+1}(\CC)\sub h_{n+1}(f^{-1}(\CC))= g^{-1}( h_n(\CC))=g^{-1}(\widetilde \CC). $$ The identity   \eqref{eq:hnfgpre} implies that  
 $(H^{n}_t)\circ  h_{n}^{-1}$ is an isotopy on $\widetilde S^2$ rel.\ $g^{-1}(\post(g))$. It   isotopes $\widetilde \CC=h_n(\CC)\sub g^{-1}(\widetilde \CC)$ into $J=h_{n+1}(\CC)$ rel.\ $g^{-1}(\post(g))$. 
So  $\widetilde{\CC}$ and $J$ are Jordan curves contained in the $1$-skeleton  $g^{-1}(\widetilde \CC)$  of $\DD^1(g,\widetilde \CC)$ that 
are isotopic relative to the set $g^{-1}(\post(g))$ of vertices of 
 $\DD^1(g,\widetilde \CC)$.  Lemma~\ref{lem:isoJcin1ske} implies that $J=\widetilde \CC$, and 
\eqref{eq:varccc} follows. 

Now  \eqref{eq:varccc} and \eqref{eq:hnsame}  imply that a chain of $n$-tiles for $(f,\CC)$ joins opposite sides of 
$\CC$ if and only if their  images under $h_n$  form a chain joining opposite sides of $\widetilde \CC$. 
Since  the images of the $n$-tiles for $(f,\CC)$ under $h_n$ are precisely the $n$-tiles 
for $(g, \widetilde \CC)$, we have $D_n(f, \CC)=D_n(g, \widetilde \CC)$  for each $n\in \N_0$. The statement   follows.   \end{proof} 
 
  Now we can show the desired  independence of 
  combinatorial expansion from the realization of a two-tile subdivision rule.

 \begin{lemma} 
   \label{lem:invrealize}
   \index{combinatorially expanding}
   \index{expanding!combinatorially}
   \index{Thurston map!combinatorially expanding}

   Let $(\DD^1,\DD^0,L)$ be a two-tile subdivision rule on $S^2$ and
   $\CC$ be the Jordan curve of $\DD^0$.  Suppose that  the maps
   $f\:S^2\ra S^2$ and $g\:S^2\ra S^2$ both realize the subdivision
   rule and that $\#\post(f)=\#\post(g)\ge 3$.  
   Then $f$ is combinatorially expanding for $\CC$ if and only if  $g$
   is combinatorially expanding for $\CC$.  
 \end{lemma} 
 


\begin{proof}  Let ${\bf V}^0$ and 
${\bf V}^1$ be the set of vertices of $\DD^0$ and $\DD^1$,
respectively.   Then $P\coloneqq \post(f)=\post(g)\sub   {\bf V}^0 \sub {\bf V}^1$.

It follows from the proof of the uniqueness part of Proposition~\ref{prop:thurstonex} 
that there exists a homeomorphism $h_1\:S^2\ra S^2$ isotopic to $\id_{S^2}$ rel.\ ${\bf V}^1\supset \post(f)=\post(g)$  that satisfies $f=g\circ h_1$.  Moreover,  $h_1(e)=e$ for each 
edge $e$ in $\DD^1$.  Since $\DD^1$ is a refinement of $\DD^0$ and so the $1$-skeleton $\CC$ of $\DD^0$ is contained in the $1$-skeleton of $\DD^1$, this   implies $h_1(\CC)=\CC$. Define $h_0=\id_{S^2}$.  Since $h_1$ is isotopic to $\id_{S^2}$ rel.\ $P$ we have  $h_1|P=\id_{S^2}|P=h_0|P$.    
Moreover, $h_0\circ f=g\circ h_1$,  $h_1(\CC)=\CC=h_0(\CC)$,  and both $h_0$ and $h_1$ are orientation-preserving homeomorphisms on $S^2$.
This shows that the hypotheses of Lemma~\ref{lem:cexp_Cinv} are satisfied (with $\widetilde S^2=S^2$), and so $f$ is combinatorially expanding for $\CC$ if and only if $g$ is combinatorially expanding for $\widetilde \CC=h_0(\CC)=h_1(\CC)=\CC$.  
\end{proof}

The previous lemma motivates   the following definition.

\begin{definition} 
  [Combinatorially expanding two-tile subdivision rules] 
  \label{def:combexprule}   
  \index{two-tile subdivision rule!combinatorially expanding|textbf}\index{combinatorially expanding!two-tile subdivision rule|textbf}
  Let $(\DD^0,\DD^1,L)$ be a two-tile subdivision rule, and $\CC$ be
  the Jordan curve of $\DD^0$. We call $(\DD^0,\DD^1,L)$ {\em
    combinatorially expanding} if every Thurston map $f$ realizing
  $(\DD^0,\DD^1,L)$ is combinatorially expanding for $\CC$.
\end{definition}
We know that if this  condition is true for {\em one} Thurston map
realizing the subdivision rule, then  it is  true for {\em all} such
maps by Lem\-ma~\ref{lem:invrealize}. 
We will see that under  an additional mild  technical assumption a two-tile subdivision rule 
can be realized by an expanding Thurston map if and only if  
the subdivision rule is combinatorially expanding 
(see Theorem~\ref{thm:combexp2}).  

  We conclude this section by proving  a statement related to Remark~\ref{rem:two-tile_fV0}~(ii).

\begin{prop} 
  \label{prop:subandD^n}
  \index{two-tile subdivision rule}\index{subdivision}
   \index{cell!complex!isomorphism}
   \index{isomorphism!of cell complexes}
 Let $(\DD^1,\DD^0,L)$ be a two-tile subdivision rule on $S^2$ and
  let $\CC$ be the Jordan curve of $\DD^0$.  Suppose that  the maps
   $f\:S^2\ra S^2$ and $g\:S^2\ra S^2$ both realize the subdivision
   rule. Let   ${\bf V}^0$ be the vertex set of $\DD^0$, and  assume 
   that  $\post(f)=\post(g)={\bf V}^0$.

  Then there exist homeomorphisms $h_n\: S^2 \ra  S^2$ for $n\in \N_0$ 
  with the following properties: for all $k,n\in \N_0$ with $ n\ge k$ we have 
  $h_n(c)\in\widetilde \DD^k\coloneqq  \DD^k(g , \CC)$ whenever $c\in 
 \DD^k\coloneqq  \DD^k(f, \CC)$. Moreover,  the map $$c\in \DD^k \mapsto h_n(c)\in 
 \widetilde \DD^k$$ is an isomorphism between the cell decompositions $\DD^k$ and $\widetilde \DD^k$ that does not depend on $n\ge k$.   
 \end{prop}
 
 We added the assumption $ \post(f)=\post(g)={\bf V}^0$ for convenience,
  because then one does not have to worry about vertices in $\DD^0$ without dynamical relevance (see Remark~\ref{rem:two-tile_fV0}~(i)).

We know that under the given assumptions $\CC$ is invariant under $f$ and $g$. So 
     in the sequence $\DD^0, \DD^1, \DD^2, \dots  $ each cell decomposition is a refinement of the previous one. Similarly, 
  $\widetilde \DD^0, \widetilde \DD^1,  \widetilde \DD^2,  \dots  $ 
  forms a sequence of finer and finer 
     cell decompositions. 
      By the proposition there exists a homeomorphism  $h_n$ that  induces  isomorphisms of the cell decompositions in the first  sequence with the corresponding cell decompositions in the second one up to level $n$.
  Moreover,  the isomorphism
 between $\DD^k$ and $\widetilde \DD^k$  given by 
 $h_n$  actually does not depend on $n\ge k$.  
 Based on this, one can easily show  the combinatorics of the sequences $\DD^n$ and $\widetilde \DD^n$, $n\in \N_0$, are exactly the same (as formulated more precisely in   
 Remark~\ref{rem:two-tile_fV0}~(ii)). 
 
 Note that in general there is no single homeomorphism $h$ on $S^2$ that induces an isomorphism between $\DD^k$ and $\widetilde \DD^k$ for {\em all} 
 levels $k\in \N_0$.  For example, no such $h$ can exist if one of the Thurston maps  is expanding, but the other one is not (for a specific case, see Example~\ref{ex:barycentric}).

\begin{proof} [Proof of Proposition~\ref{prop:subandD^n}]  
As in the statement, we use the notation $\DD^i\coloneqq \DD^i(f,\CC)$ and $\widetilde \DD^i\coloneqq \DD^i(g,\CC)$ for $i\in \N_0$. For $i=0,1$  the first definition is consistent with our notation $\DD^0$ and $\DD^1$ for the cell decompositions of the subdivision rule, because  
under our assumptions we have $\DD^0=\DD^0(f,\CC)=\DD^0(g, \CC)=\widetilde \DD^0$ and $\DD^1=\DD^1(f,\CC)=\DD^1(g,\CC)=\widetilde \DD^1$. 
Note that $\#\V^0\ge 3$ by our definition of a two-tile subdivision rule. 

By the proof of 
Lemma~\ref{lem:invrealize}, there are
orientation-preserving homeo\-mor\-phisms 
$h_0=\id_{S^2}$ and $h_1\: S^2  \ra S^2$ with 
$\CC=h_0(\CC)=h_1(\CC)$ satisfying the hypotheses 
of Lemma~\ref{lem:cexp_Cinv} (with $\widetilde S^2=S^2$).
Moreover, both $h_0$ and $h_1$ fix the points in $ \post(f)=\post(g)={\bf V}^0$. It follows from the considerations  in the proof of Lemma~\ref{lem:cexp_Cinv} that we obtain  orientation-preserving  homeomorphisms $h_n\: S^2\ra S^2$ for all  $n\in \N_0$  that fix the curve $\CC$ as a set and each point  in $\V^0$. 
This implies that each $h_n$ fixes all cells in $\DD^0$ as sets.
  
 As we have seen in the proof of Lemma~\ref{lem:cexp_Cinv}, these homeomorphisms  also satisfy 
\begin{equation}\label{eq:hnfgn+1}
h_n\circ f=g\circ h_{n+1}
\end{equation}
for $n\in \N_0$. 

Now let  $n,k\in \N_0$ with $k\le n$ be arbitrary, and consider the cell decomposition 
$$h_n(\DD^k)\coloneqq \{h_n(c'): c'\in \DD^k\}. $$
It follows from repeated application  of \eqref{eq:hnfgn+1}  
that  $g^k=h_{n-k}\circ f^k\circ h_n^{-1}$. 
This implies that if $c'\in \DD^k$, then $g^k$ is a homeomorphism 
of $h_n(c')\in h_n(\DD^k) $ onto $h_{n-k}(f^k(c'))$. Now $c\coloneqq f^k(c')$ is a cell in $\DD^0=\widetilde \DD^0$ and so $h_{n-k}(c)=c$. In other words, $g^k$ is cellular 
for $(h_n(\DD^k), \widetilde \DD^0)$. Since $g^k$ is also cellular for 
$(\widetilde \DD^k, \widetilde \DD^0)$, the uniqueness statement
in Lemma~\ref{lem:pullback} implies that $\widetilde
\DD^k=h_n(\DD^k)$. 
The first part of the statement follows.

It remains to show that the isomorphism between $\DD^k$ and $\widetilde \DD^k$ induced by $h_n$ does not depend on $n\ge k$. Since we will not use this statement in the following, we provide only a sketch of the proof leaving some details to the reader. 

It is enough to show that 
\begin{equation}\label{eq:fixxDDk}
(h_{n+1}^{-1}\circ h_n)(c')=c',
\end{equation}  whenever 
$c'\in \DD^k$ and $n\ge k$. Since both $h_n$ and $h_{n+1}$ induce 
isomorphisms of $\DD^k$ and $\widetilde \DD^k$, we know that  
$(h_{n+1}^{-1}\circ h_n)(c')\in \DD^k$ for each $c'\in \DD^k$.  

Now \eqref{eq:fixxDDk}  is true if $c'$ consists of a vertex
in $\DD^k$, i.e., a point in $f^{-k}(\post(f))$. This follows from the fact that 
$h_{n+1}$ and $h_n$ are actually isotopic rel.\ $f^{-n}(\post(f))$, and so 
$h_{n+1}^{-1}\circ h_n$ fixes each point in $f^{-n}(\post(f))\supset 
f^{-k}(\post(f))$. 

Since $h_{n+1}^{-1}\circ h_n$ is isotopic to $\id_{S^2}$ rel.\
$f^{-n}(\post(f))\supset f^{-k}(\post(f))$, the argument in the proof of Lemma~\ref{lem:isoJcin1ske}
shows that \eqref{eq:fixxDDk} is also valid for each  edge $c'$  in 
$\DD^k$. This in turn  implies that  if $X\in \DD^k$ is a tile,
then  $(h_{n+1}^{-1}\circ h_n)(X)$ is a tile in $\DD^k$ with the
same boundary as $X$.  Since $h_{n+1}^{-1}\circ h_n$ is orientation-preserving and fixes the vertices  on $\partial X$, we conclude that  $(h_{n+1}^{-1}\circ h_n)(X)=X$. Equation~\eqref{eq:fixxDDk} follows. 
\end{proof}

  \section{Examples of two-tile subdivision rules}
\label{sec:examples-two-tile}

In this section we present some examples of two-tile subdivision rules 
$(\DD^1,\DD^0, L)$ and maps that realize them.  This is based 
on Proposition~\ref{prop:rulemapex}.  We have already used this method for constructing and describing 
Thurston maps before (see the remark after Proposition~\ref{prop:rulemapex}).

In most of our examples we will represent the underlying sphere
$S^2$ as a pillow $P$ (see Section~\ref{sec:expratThmaps})
obtained by gluing together two  isometric copies $X^0_{\tt w}$ and $X^0_{\tt b}$  of a (simple) Euclidean  polygon $X\sub \C$. This gives us a natural cell decomposition 
$\DD^0$ of $S^2$, where $X^0_{\tt w}$ and $X^0_{\tt b}$  are the $0$-tiles, and the sides and corners on the common boundary of the polygons the $0$-edges and $0$-vertices. We assign the color ``white'' to  $X^0_{\tt w}$, and ``black'' to $X^0_{\tt b}$. In our figures the top polygon of the pillow will  be the white $0$-tile. 
The  pillow $P$ carries a natural orientation (as discussed in 
Section~\ref{sec:expratThmaps}). 
This in turn determines an orientation of  the equator $\CC\coloneqq \partial 
  X^0_{\tt w}=\partial 
  X^0_{\tt b}$ of $P$ so that $X^0_{\tt w}$ lies on the left and 
  $X^0_{\tt b}$ on the right of $\CC$
  (see Section~\ref{sec:orient}).  
  
The description of the cell decomposition $\DD^1$ is usually more complicated and depends on the specific case.

We know that in order to uniquely specify the labeling $L\: \DD^1\ra \DD^0$ it is enough to know the image of a pair $(v', X')$, where $X'$ is a $1$-tile  
and $v'\in X'$ a $1$-vertex 
(see Lemma~\ref{lem:labeluniq}~\ref{item:labeluniq2}). In general,  we will include more information on the labeling $L$ for a better illustration of the 
behavior of the map $f$ realizing the subdivision rule. We will
assign 
the colors ``black'' or ``white'' 
to the $1$-tiles indicating to which of the $0$-tiles 
they are sent by $L$ (and $f$). With these labels the cell decomposition $\DD^1$ will be a checkerboard tiling of $k$-gons, where $k$ is the number of vertices in $\DD^0$.

 Sometimes we will introduce markings for  the $0$-vertices and some  $1$-vertices,  often suggested by a natural identification 
 of the underlying sphere $S^2$ with the Riemann sphere $\CDach$. 
For a $1$-vertex marked $a$, we indicate by ``$a\mapsto b$'' that
the labeling $L$ sends it to the $0$-vertex marked $b$.
Similarly, ``$\mapsto b$'' indicates a $1$-vertex without
additional marking that is sent by $L$ to the $0$-vertex marked
$b$.

  After these preliminaries we now proceed to discussing the examples.

\ifthenelse{\boolean{nofigures}}{}{
\begin{figure}
  \centering
  \begin{picture}(10,10)
    \put(258,62){$\scriptstyle{0}$}
    \put(210,62){$\scriptstyle{-1}$}
    \put(264,107){$\scriptstyle{\infty}$}
    \put(86,58){$\scriptstyle{0\mapsto -1}$}
    \put(126,58){$\scriptstyle{1\mapsto 0}$}
    \put(88,107){$\scriptstyle{\infty\mapsto \infty}$}
    \put(36,44){$\scriptstyle{-1\mapsto 0}$}
    \put(167,78){$\scriptstyle{f}_1$}
  \end{picture}
  \includegraphics[width=11cm]{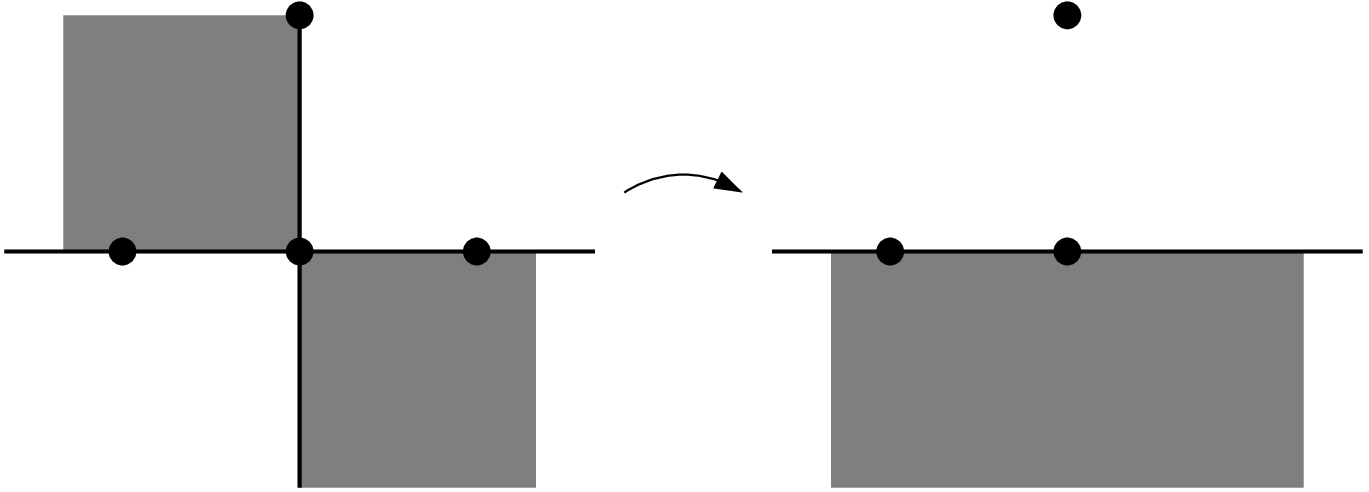}
  \caption{The two-tile subdivision rule for $z^2-1$.}
  \label{fig:subdivz2-1}
\end{figure}
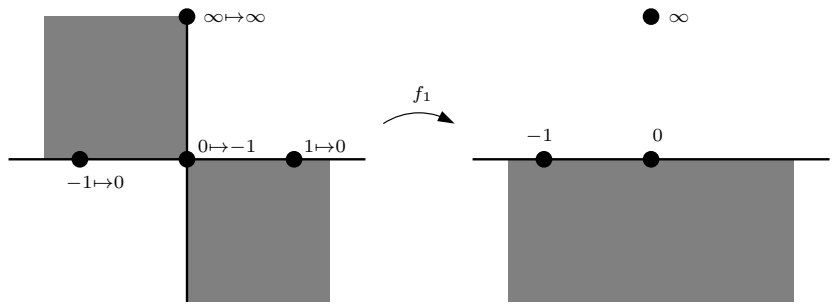

}

\begin{ex}
\label{ex:z2-1}
 Here the white $0$-tile is the (closure
  of the) upper half-plane, and the black $0$-tile is the (closure of the)
  lower half-plane in $\CDach$. The $0$-vertices are the points
  $-1,0,\infty$. Thus the $0$-edges are $[-\infty, -1], [-1,0],
  [0,\infty]\sub \widehat \R$. The cell decomposition $\DD^0$ is indicated on the
  right in  Figure \ref{fig:subdivz2-1}.

  The white $1$-tiles are the first and third quadrants, and the black
  $1$-tiles are the second and forth quadrants. The $1$-vertices and their labelings  are as
  follows: the point $\infty$ is the only $1$-vertex labeled $\infty$,
  the $1$-vertices $-1$ and $1$ are labeled $0$, the $1$-vertex $0$ is
  labeled $-1$. The cell decomposition is indicated on  the left in
  Figure \ref{fig:subdivz2-1}. 

 This defines  an orientation-preserving labeling $L$. Then
$(\DD^1,\DD^0, L)$ is a two-tile subdivision rule.
 It is straightforward to check that  the map $f_1(z)=z^2-1$ realizes the subdivision rule   $(\DD^1,\DD^0, L)$. 
  
\begin{figure}
  \centering
  \includegraphics[width=12cm]{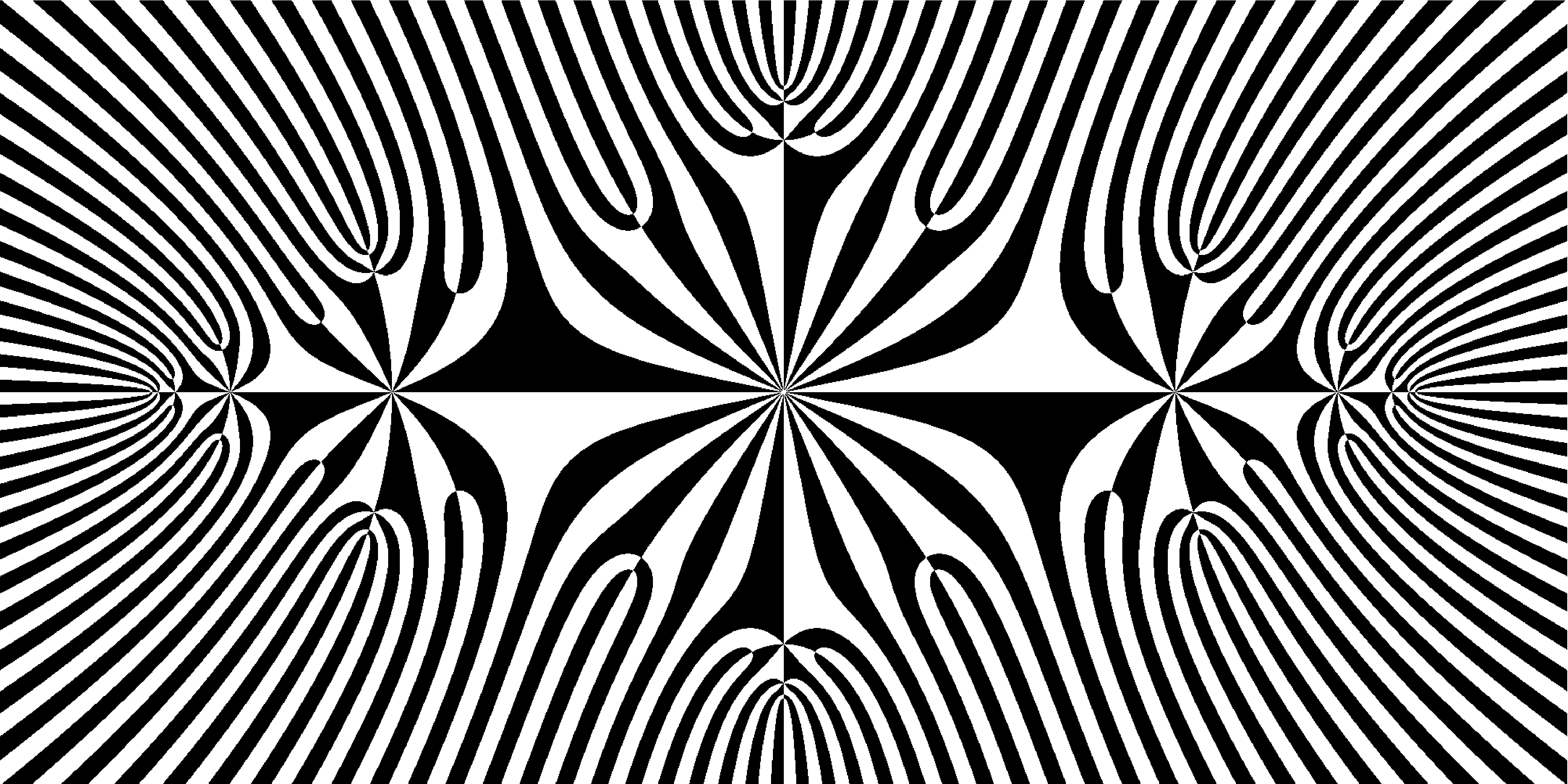}
  \caption{Tiles of  level $7$ for  Example \ref{ex:z2-1}.}
  \label{fig:tiles8_z2-1}
\end{figure}

 Since $f_1$ is a Thurston polynomial, it cannot be expanding by 
  Lemma~\ref{lem:poly_not_exp}.
  In fact, $[-1,0]$ is an $n$-edge for each  $n\in \N_0$ and so there exist two $n$-tiles that contain all postcritical points
  $-1,0,\infty$. 
  Figure \ref{fig:tiles8_z2-1} shows the tiles of  level $7$.  
\end{ex}

\begin{ex}[The barycentric subdivision rule] 
  \label{ex:barycentric}
  \index{barycentric subdivision}\index{subdivision!barycentric}
  We glue two equilateral triangles together along their
  boundaries to form a pillow  $S^2$. It  is a polyhedral surface and so 
  conformally equivalent to $\CDach$.  The two triangles are the
  $0$-tiles. We can find a conformal equivalence of $S^2$ with
  $\CDach$ such that the triangles correspond to the upper and
  lower half-planes, and the vertices to the points
  $-1,1,\infty$. For convenience we identify the vertices with
  $-1,1,\infty$; they are the $0$-vertices. The $0$-edges are the
  three edges of the triangles. The bisectors divide each
  triangle (each $0$-tile) into six smaller triangles. These $12$
  small triangles are the $1$-tiles. The labeling of the
  $1$-vertices is indicated in Figure \ref{fig:barycentric}.
  Again we obtain a two-tile subdivision rule.  We can realize
  this subdivision rule by a map $f_2$ that {\em conformally}
  maps $1$-tiles to the $0$-tiles. Under the indicated
  identification of $S^2$ with $\CDach$, the map is given by
    \begin{equation*}
    f_2(z)=1- \frac{54 (z^2-1)^2}{(z^2+3)^3}
  \end{equation*}
  (see \cite[Example 4.6]{CFKP}). The subdivision rule is
  combinatorially expanding, but the map $f_2$ is not
  expanding. This follows from
  Proposition~\ref{prop:rationalexpch}, because the point $1$ is
  both a critical and a fixed point of $f_2$. In
  Figure~\ref{fig:tiles_bary} the tiles of levels $1$--$6$ are
  shown. The tiles intersecting the borders of a picture frame  are actually unbounded; 
  these tiles form  the (closures of the) flowers
  at  $\infty$. 
 One can show that for a fixed $n$-vertex $v$, the intersection
  of the $m$-flowers $W^m(v)$, for $m\geq n$, is not a single
  point, but in fact the closure of the Fatou component of $f_2$ 
  containing $v$. 
 The Julia set of $f_2$ is a Sierpi\'{n}ski carpet, i.e., a set
  homeomorphic to the standard Sierpi\'{n}ski carpet fractal. 

\ifthenelse{\boolean{nofigures}}{}{
\begin{figure}
  \centering
  \includegraphics[width=11cm]{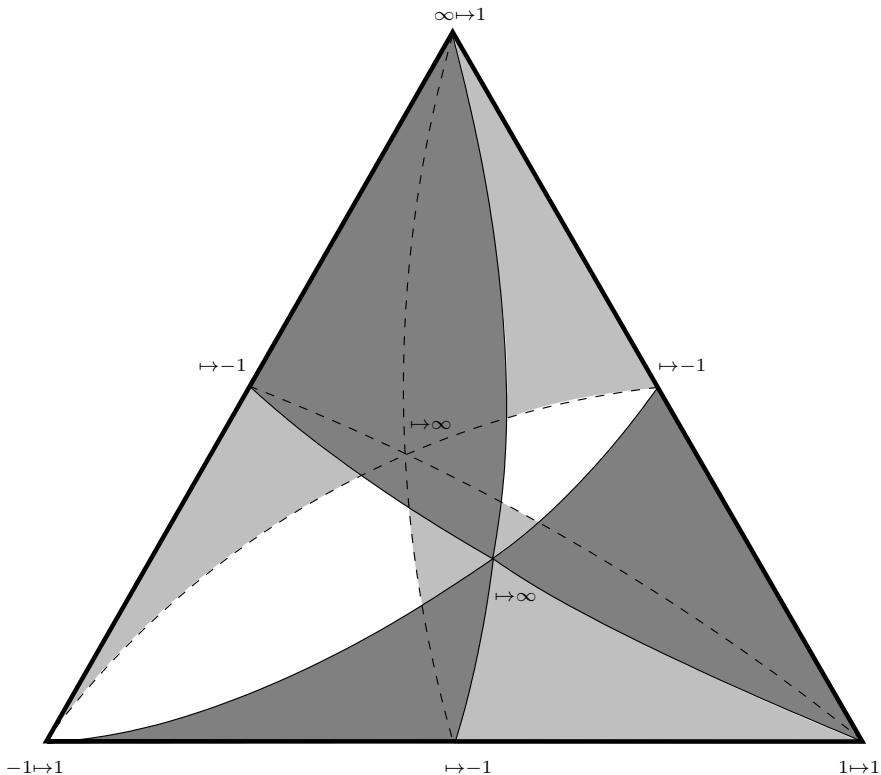}
  \begin{picture}(10,10)
    \put(-330,-10){$\scriptstyle{-1\mapsto 1}$}
    \put(-164,-10){$\scriptstyle{\mapsto -1}$}
    \put(-15,-10){$\scriptstyle{1\mapsto 1}$}
    \put(-168,275){$\scriptstyle{\infty\mapsto 1}$}
    \put(-257,142){$\scriptstyle{\mapsto -1}$}
    \put(-83,142){$\scriptstyle{\mapsto -1}$}
    \put(-177,120){$\scriptstyle{\mapsto \infty}$}
    \put(-145,55){$\scriptstyle{\mapsto \infty}$}
  \end{picture}
  \caption{The barycentric subdivision rule.}
  \label{fig:barycentric}
\end{figure}
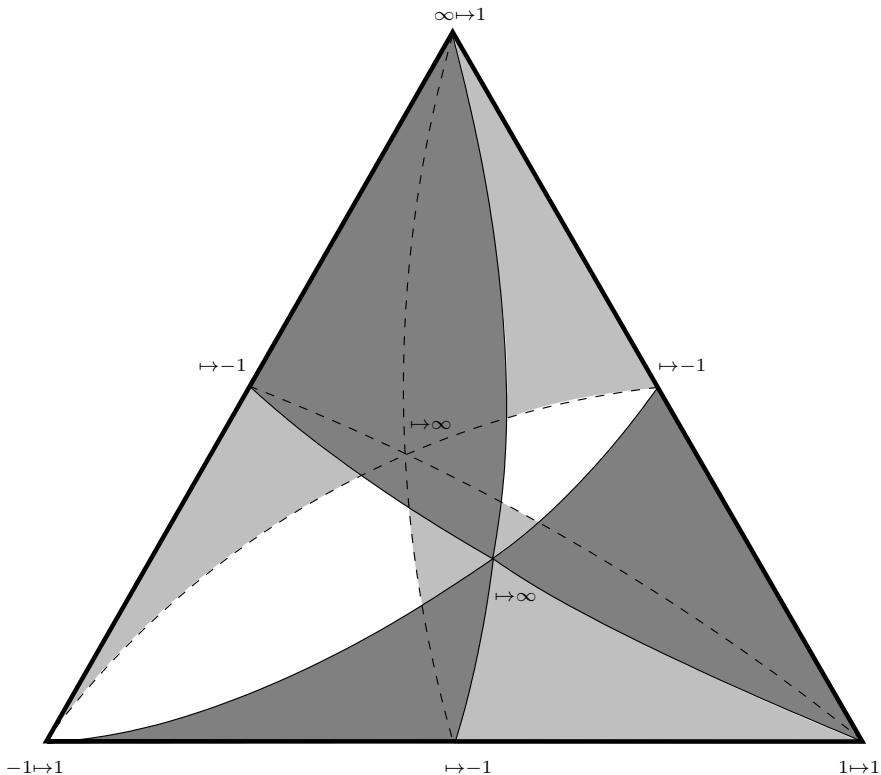
}

  It is possible to choose a different realization of the
  two-tile subdivision rule indicated in Figure \ref{fig:barycentric}
  by a map $\widetilde{f}_2$ that is expanding.  Namely, we can use
  {\em affine} maps to map the $1$-tiles (the small triangles in
  the barycentric subdivision  of the equilateral triangles)
  to the $0$-tiles.  In this case, the $n$-tiles are Euclidean
  triangles for each $n\in \N_0$. The collection of all $n$-tiles
  is obtained from the $(n-1)$-tiles as the $1$-tiles
  were constructed from the $0$-tiles: one subdivides each
  Euclidean triangle representing an $(n-1)$-tile by its
  bisectors. It is not difficult to see that the diameters
  of $n$-tiles tend to 
  $0$ as $n\to \infty$. Hence $\widetilde{f}_2$ is expanding, and
  so this map is an example of an expanding Thurston map with
  periodic critical points.

\ifthenelse{\boolean{nofigures}}{}{
\begin{figure}
  \centering
  \begin{subfigure}{0.48\textwidth}
    \begin{overpic}
      [width=2.395in, 
      tics=20]{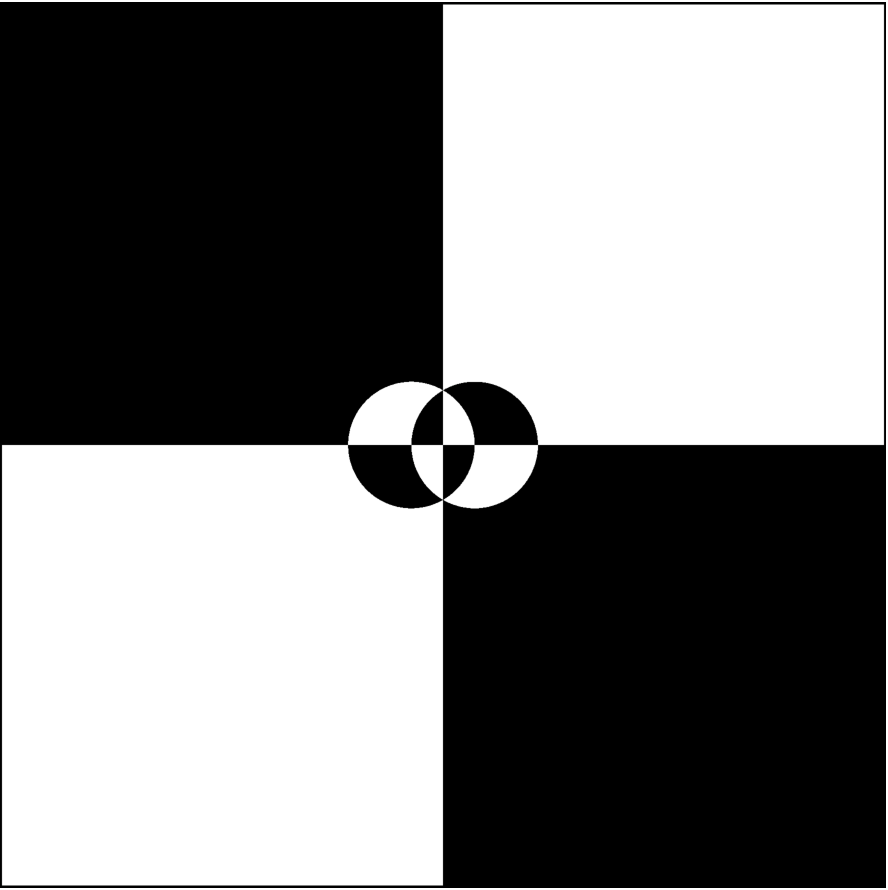}
    \end{overpic}
  \end{subfigure}
  \hspace*{\fill}
  \begin{subfigure}{0.48\textwidth}
    \begin{overpic}
      [width=2.395in, 
      tics=20]{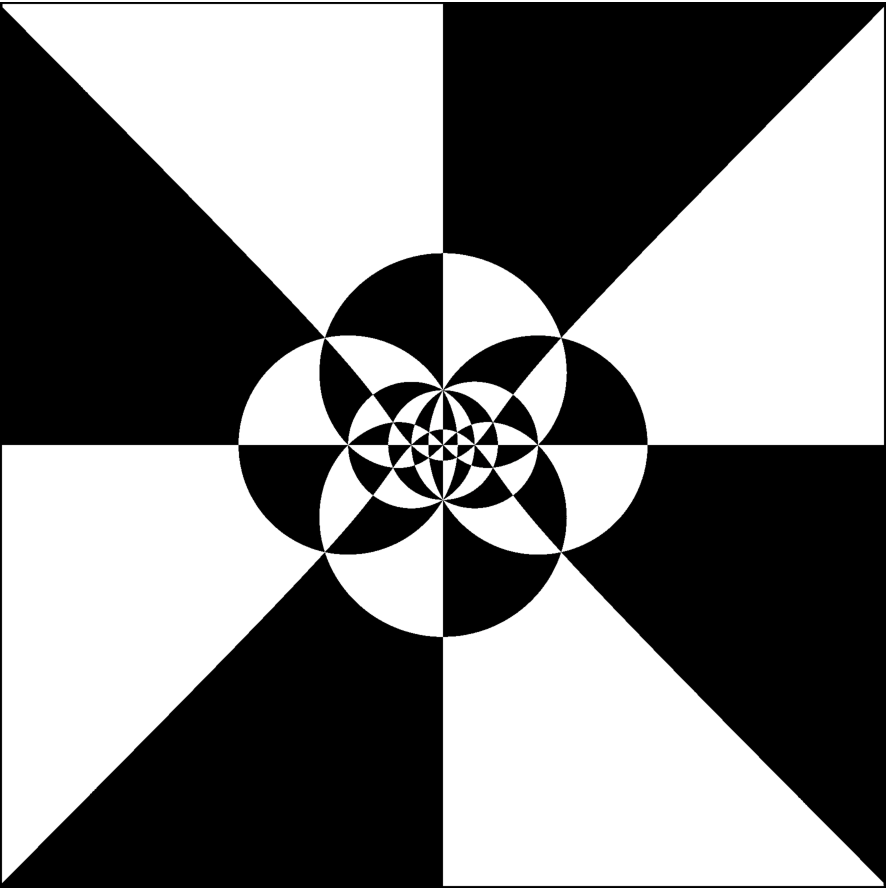}
    \end{overpic}
  \end{subfigure}
  
  \vspace{0.04\textwidth}  
  \begin{subfigure}{0.48\textwidth}
    \begin{overpic}
      [width=2.395in, 
      tics=20]{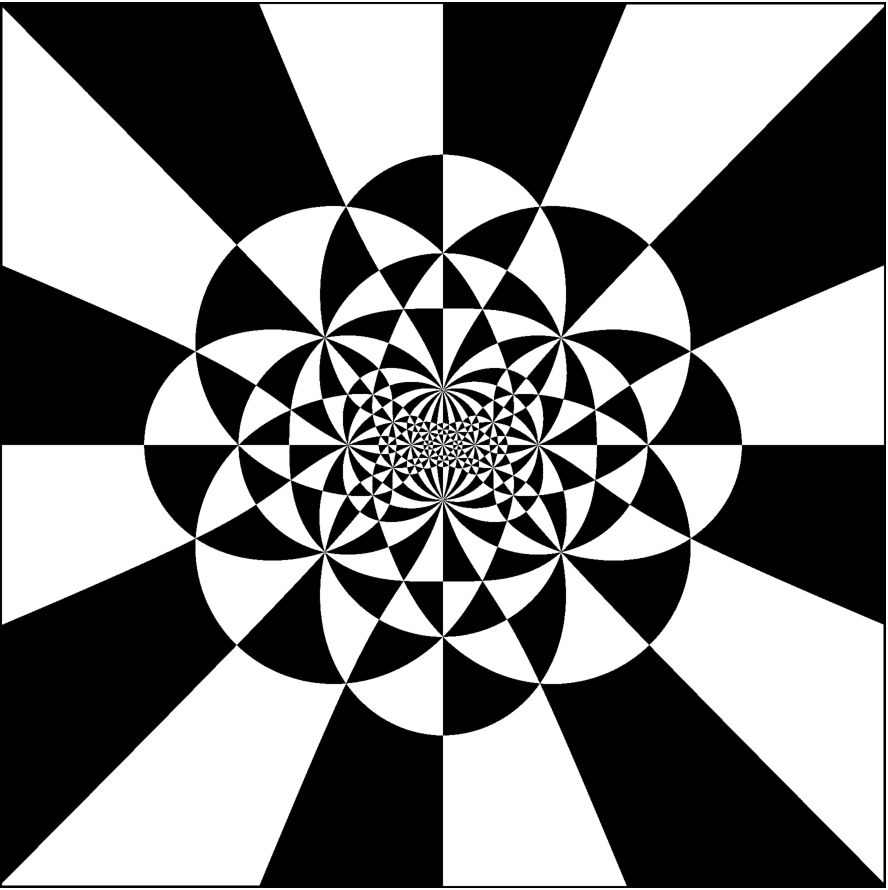}
    \end{overpic}
  \end{subfigure}
  \hspace*{\fill}
  \begin{subfigure}{0.48\textwidth}
    \begin{overpic}
      [width=2.395in, tics=20]{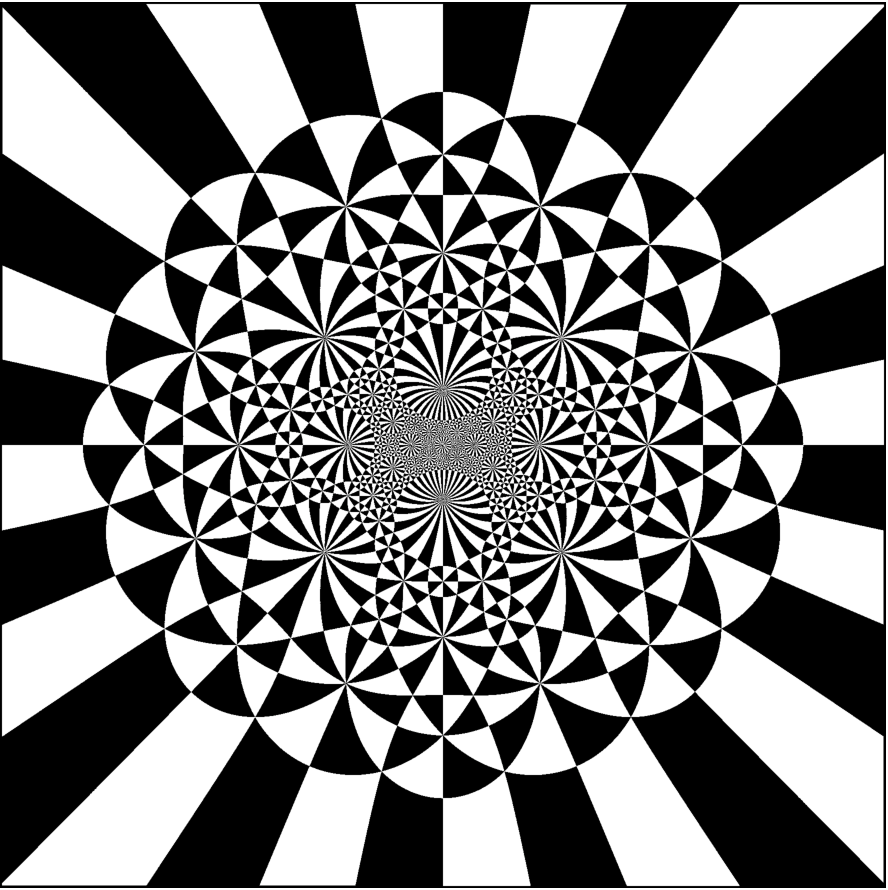}
    \end{overpic}
  \end{subfigure}
  
  \vspace{0.04\textwidth}
  \begin{subfigure}{0.48\textwidth}
    \begin{overpic}
      [width=2.395in, tics=20]{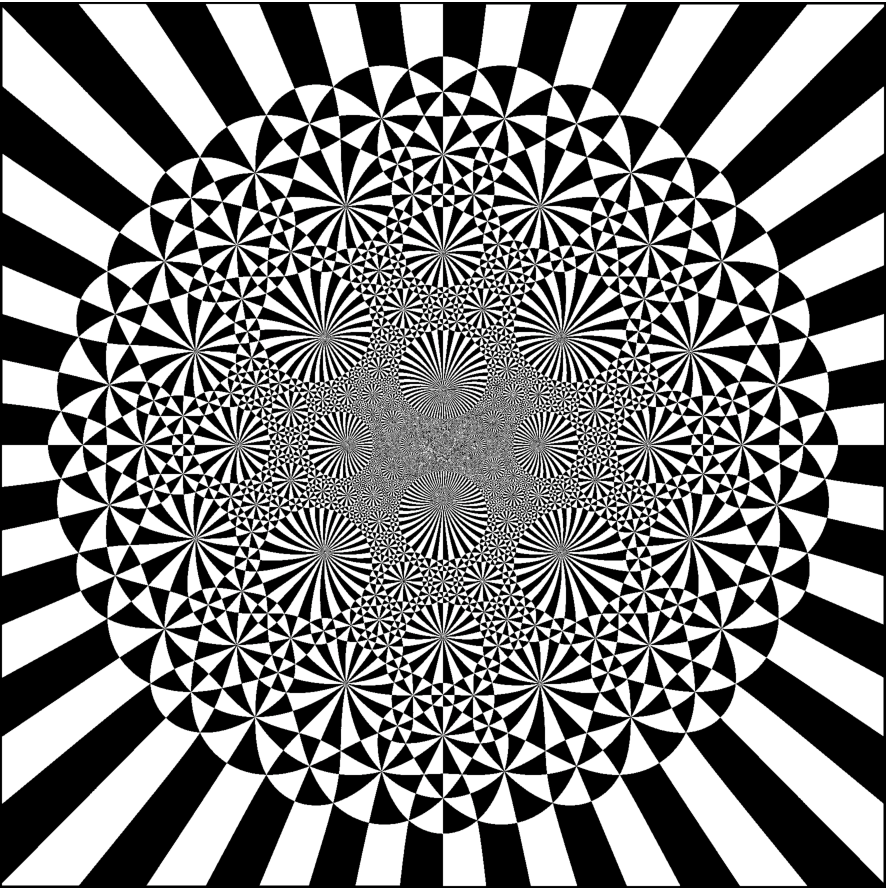}
      \end{overpic}
    \end{subfigure}
    \hspace*{\fill}
  \begin{subfigure}{0.48\textwidth}
    \begin{overpic}
      [width=2.395in, tics=20]{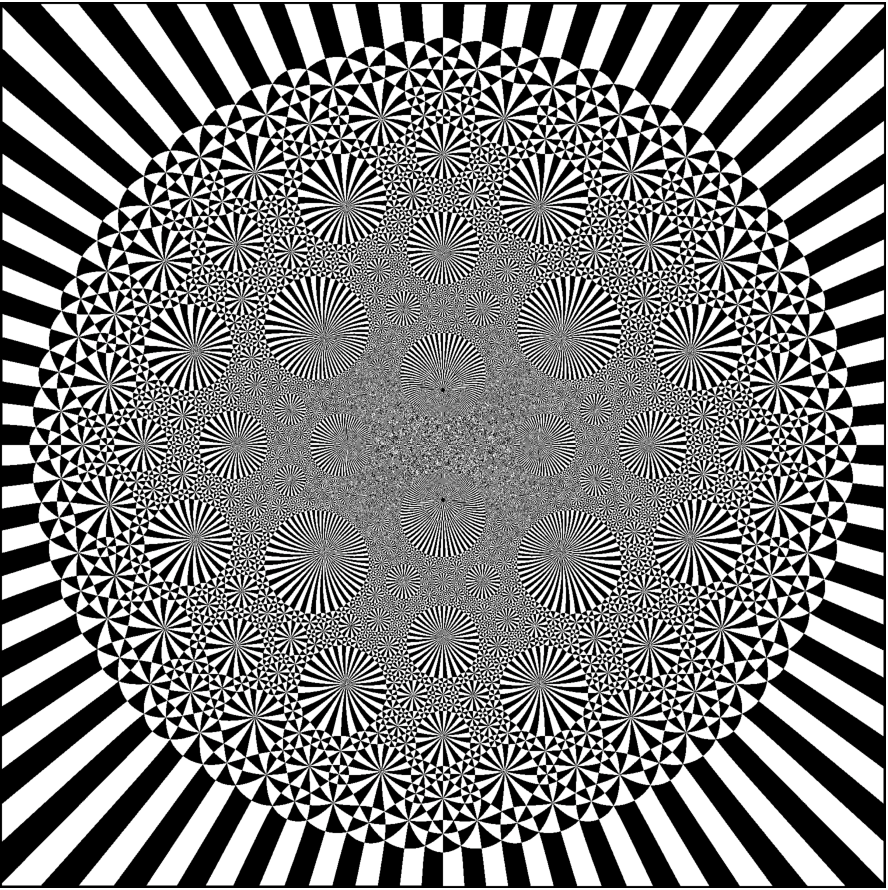}
    \end{overpic}
  \end{subfigure}
  \caption{Tiles of levels 1--6 for the barycentric subdivision map $f_2$.}
  \label{fig:tiles_bary}
\end{figure}
}

  In Chapter~\ref{cha:combexp} we will present a general procedure how to obtain an
  expanding Thurston map from a combinatorially expanding one.
 Roughly speaking, we define an equivalence relation on the underlying sphere 
 that collapses sets where the map fails to be
  expanding to a point. For this example $f_2$,  these  equivalence classes are given by  the sets
of the form  $\bigcap_{m\geq n} W^m(v)$, where
$v$   is an $n$-vertex.  
\end{ex}

\begin{ex}
  \label{ex:obstructed_map}
  The map $f_3=h$ constructed in Section~\ref{sec:int-frac-sph}
  realizes the two-tile subdivision rule shown in
  Figure~\ref{fig:1flap_both}. It is obvious that this subdivision rule 
  is combinatorially expanding.  We  revisited the 
  map $f_3$ in Example~\ref{ex:obstructedThmap}, where we saw that
  it  is not Thurston equivalent to a rational map.

This map is related  to the  Latt\`{e}s map $g$
considered 
in Section~\ref{sec:Lattes}. Namely,  we cut the sphere along one
$1$-edge of $g$, and glued in two small squares that are mapped 
to the $0$-tiles. This increased the
degree of the map by $1$. A similar construction is possible in
greater generality. This was introduced by Pilgrim and Tan Lei, who called 
this operation ``blowing up an arc'' (see \cite{PT98}). 
\end{ex}

\begin{ex}
  [The 2-by-3 subdivision rule]
  \label{ex:2x3}
  We present another example of an expanding Thurston map $f_4$
  that is not (Thurston) equivalent to a rational map. In a
  sense, 
  this is the easiest example of this type, 
  but it has a parabolic
  orbifold in contrast to the previous one.

  The map $f_4$ is a Latt\`{e}s-type map (see
  Definition~\ref{def:Lattestype}) with signature $(2,2,2,2)$
as provided  by Proposition~\ref{prop:2222}. Here the map
  $A\: \C\ra \C$ in \eqref{eq:Lattestype2222} is given by
  $A(x+ y \iu ) = 2x+ 3 y \iu $ for $x,y\in \R$ and the associated crystallographic group  $G$ consists 
   of all isometries on $\R^2\cong \C$ of the form
  $g(u)= \pm u +\gamma$, where $\gamma\in \Z^2$. Then $A$
  descends to 
  the map
  $f_4\:S^2\ra S^2$ 
  on the quotient
  $S^2=\R^2/G$. As discussed in Example~\ref{ex:lattes_type}, we
  can represent the quotient space by a pillow $P$ obtained by
  gluing two squares of side length $1/2$ together along their
  boundaries. We have seen in Section~\ref{sec:Lattes} that this
  pillow can naturally be identified with $\CDach$ via a map that is essentially 
a   Weierstra\ss\ $\wp$-function (see 
  Section~\ref{sec:examples-lattes-maps}). This explains the
  markings of the four $0$-vertices in Figure~\ref{fig:2x3} which
  represents the two-tile subdivision rule realized by $f_4$.

  \ifthenelse{\boolean{nofigures}}{}{  
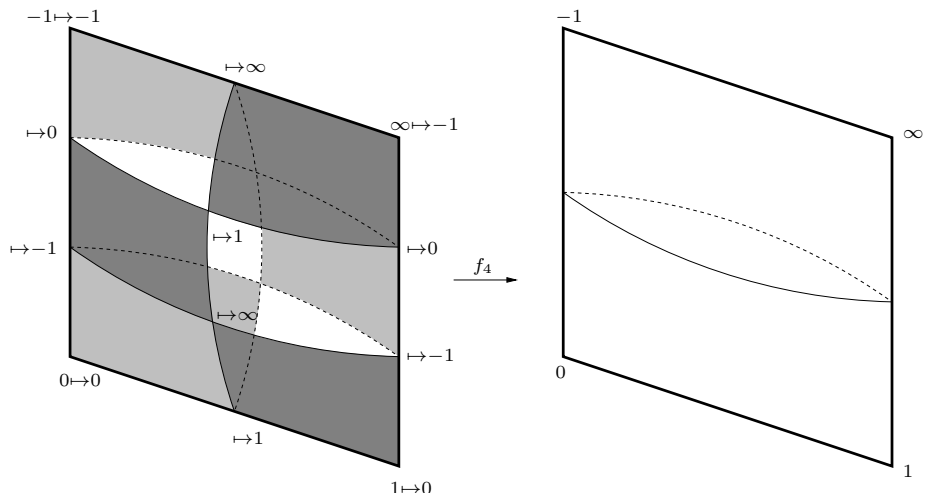
\begin{figure}
  \centering
  \begin{overpic}
    [width=11cm,  
    tics=20]{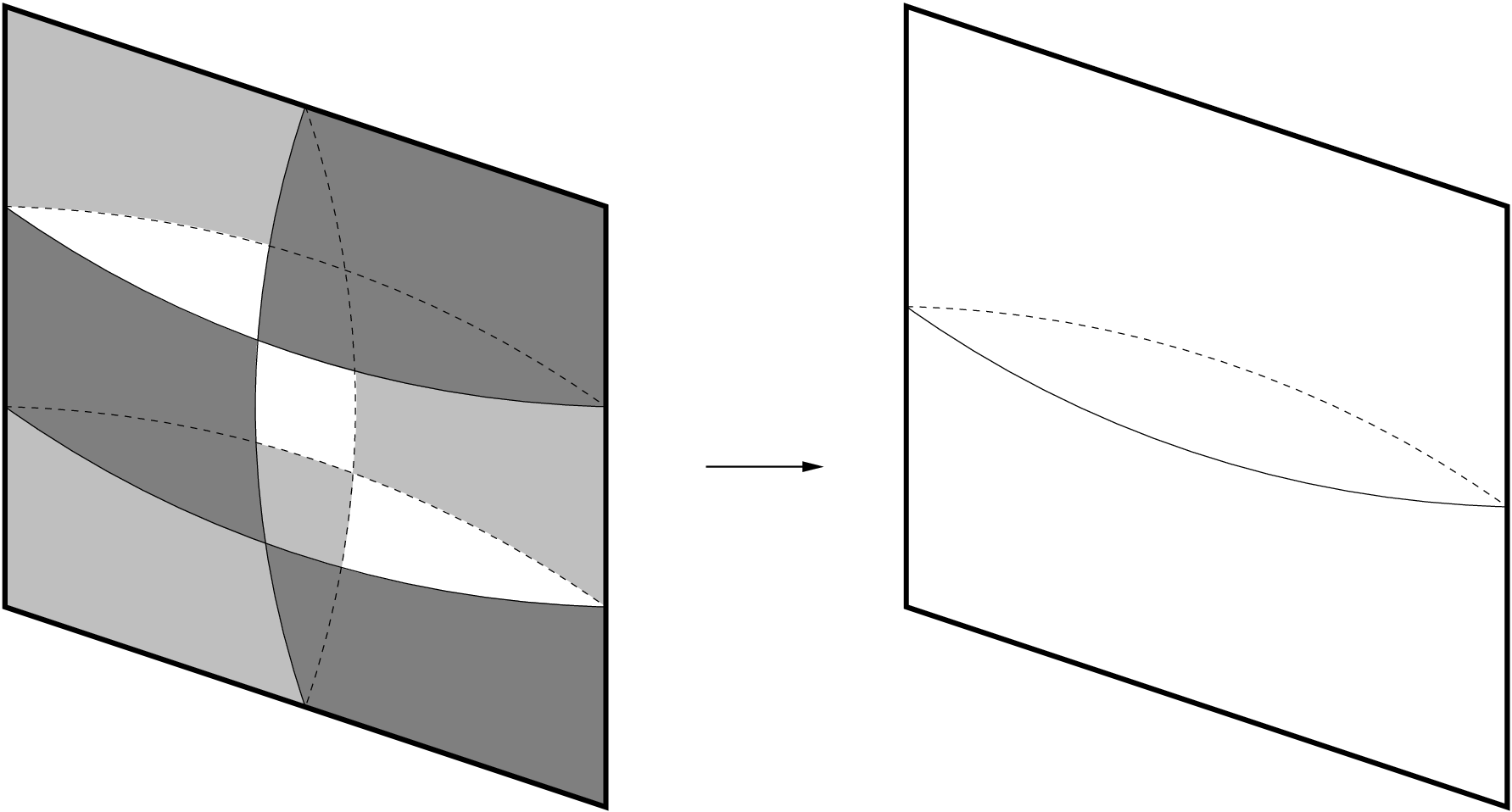}
    \put(59,11){$\scriptstyle{0}$}
    \put(101,-1){$\scriptstyle{1}$}
    \put(101,40){$\scriptstyle{\infty}$}
    \put(59,54){$\scriptstyle{-1}$}
    \put(39,-3){$\scriptstyle{1\mapsto 0}$}
    \put(20,3){$\scriptstyle{\mapsto 1}$}
    \put(-1,10){$\scriptstyle{0\mapsto 0}$}
    \put(-7,26){$\scriptstyle{\mapsto -1}$}
    \put(-5,40){$\scriptstyle{\mapsto 0}$}
    \put(-5,54){$\scriptstyle{-1\mapsto -1}$}
    \put(41,13){$\scriptstyle{\mapsto -1}$} 
    \put(41,26){$\scriptstyle{\mapsto 0}$}
    \put(39,41){$\scriptstyle{\infty \mapsto -1}$}
    \put(19,48){$\scriptstyle{\mapsto \infty}$}
    \put(18,18){$\scriptstyle{\mapsto \infty}$}
    \put(17.5,27.5){$\scriptstyle{\mapsto 1}$}
    \put(49,24){$\scriptstyle{f_4}$}
  \end{overpic}
  \caption{The 2-by-3 subdivision rule.}
  \label{fig:2x3}
\end{figure}
}

Each of the two faces (i.e., squares) of the pillow is divided
into six  rectangles as shown in Figure~\ref{fig:2x3}. These $12$
rectangles are the $1$-tiles. Their sides and vertices are the
$1$-edges and $1$-vertices. The coloring of $1$-tiles and 
the labeling of the $1$-vertices is indicated on the left in 
Figure~\ref{fig:2x3}.  The map $f_4$ sends each of the $12$
rectangles affinely to one of the two squares forming the faces
of the pillow.  This implies that each $n$-tile is a rectangle
with side lengths 
$\frac{1}{2}2^{-n}$ and $\frac{1}{2}3^{-n}$. 
In particular, $f_4$ is an
expanding Thurston map.
  
The fact that $f_4$ is not equivalent to a rational map can be derived from  Theorem~\ref{thm:Thurston_para}. In the following, we will
sketch a different  argument for this which is more in line with 
 our general  framework.  In our outline, we will rely on some results and concepts
that will be discussed later on.

To reach a contradiction, suppose that $f_4$ is equivalent to a rational map
$R\: \CDach\ra \CDach$. Then $R$ is a Thurston map with no
periodic critical points. Hence $R$ is expanding (Proposition~\ref{prop:rationalexpch}). 
So by Theorem~\ref{thm:exppromequiv} the maps $f_4$ and $R$ are topologically conjugate. This  implies by Theorem~\ref{thm:S2vsf}~\ref{item:S2qsphere} that if our  pillow $P$ is equipped with a visual metric $\varrho$
for $f_4$, then $(P, \varrho)$ is quasisymmetrically equivalent to the standard $2$-sphere, i.e.,  $\CDach$ equipped with the chordal metric. 
In particular, if $X^0$ is a $0$-tile  (i.e., 
 one of the  faces of the pillow $P$) equipped with a visual metric $\varrho$, 
 then it can be embedded  into  $\CDach$ by a quasisymmetry. 

  Now there are visual metrics for $f_4$ with expansion factor 
  $\Lambda=2$.  It is not hard to see this directly; it also follows from the general argument in the proof of Theorem~\ref{thm:visexpfactors1}  based on  \eqref{extraonL}.  Indeed, if $\CC$ is the equator of the pillow (which is $f_4$-invariant), then we  have 
  $D_1=D_1(f_4,\CC)=2$ in \eqref{extraonL}, which guarantees the existence of the desired visual metric. 
   
   If $\varrho$ is such a metric, then  $(X^0, \varrho)$ is bi-Lipschitz equivalent 
  to a Rickman's rug
  $R_\alpha$. Here by definition the 
  {\em Rickman's rug}\index{Rickman's rug} $R_\alpha$ for $0<\alpha<1$ is the unit square
  $[0,1]^2\sub \R^2$ equipped with the  metric $d_\alpha$ 
  given by 
  $$d_\alpha((x_1,y_1), (x_2,y_2)) =
  \abs{x_1-x_2} + \abs{y_1-y_2}^\alpha$$ for $(x_1,y_1), (x_2,y_2)\in 
  [0,1]^2$.  In our 
  case,  $(X^0,\varrho)$ is bi-Lipschitz equivalent to $R_\alpha$ with $\alpha= \log 2/\log 3$. It is well known that
  no quasisymmetry  can lower the
  Hausdorff dimension 
  $$\dim_H(R_\alpha)=1+ \log 3/\log 2>2$$ of $R_\alpha$ 
  (see \cite[Theorem
  15.10]{He}); in particular, $R_\alpha$ and hence also $(X^0, \varrho)$,  cannot be embedded into $\CDach$ by a quasisymmetry. 
  This is a contradiction showing that $f_4$ is not Thurston equivalent to a rational map. 
\end{ex}

\begin{ex}
  \label{ex:R_mario3}
  We now present a whole class of
  examples. In fact, the Latt\`{e}s map $z\mapsto 1-2/z^2$ from
  Example~\ref{ex:Lattes244}, the map from
  Example~\ref{ex:tringflP}, and the one from
  Example~\ref{ex:f2flaps} are all members of this family. The
  construction of these maps is illustrated in
  Figure~\ref{fig:lattes_1_2_flaps}. 

  The starting point is the Latt\`{e}s map
  $f_5(z)\coloneqq1-2/z^2$, which is the map in our family
  of lowest degree. We briefly recall the geometric description of this map
  as indicated in   Figure~\ref{fig:lattes244a}. 
For this   let $T$ be  the right-angled
  isosceles Euclidean triangle whose hypotenuse has length $1$;
  its  angles are $\pi/2, \pi/4,\pi/4$. We also consider a  smaller triangle $T'$ 
  similar to $T$  by the scaling factor $\sqrt{2}$.  We obtain 
a  pillow $\Delta$ by gluing two isometric copies
  $T_\wt$ and $T_{\bt}$ of $T$ together along their
  boundaries. The pillow carries a natural cell decomposition 
  $\DD^0$ whose $0$-tiles are $T_\wt$ and $T_{\bt}$, with 
  the common corners and sides of these triangles as $0$-vertices and $0$-edges.  A cell decomposition 
  $\DD^1$ of $\Delta$ is now obtained by subdividing   $T_\wt$ and  $T_{\bt}$  
  by the bisectors perpendicular to their hypotenuses  into two triangles each. Then $\DD^1$ contains four $1$-tiles    
  isometric to $T'$. If we choose a labeling as indicated 
 at the top in  Figure~\ref{fig:lattes_1_2_flaps} (corresponding to 
 Figure~\ref{fig:lattes244a}), then we obtain a two-tile subdivision rule $(\DD^1,\DD^0, L)$. It can be realized by a Thurston map
 $g=g_5\colon \Delta\to \Delta$ that sends each of the four small triangles 
 to $T_\wt$ or   $T_{\bt}$ by a suitable similarity.  

  Since $\Delta$ is a polyhedral surface, it can naturally be
  viewed as a Riemann surface. By the uniformization 
  theorem there is a conformal map 
  $\varphi\colon \Delta \to \CDach$. It can be chosen so that its sends the $0$-vertices of $\Delta$ (i.e., the common corners of 
  $T_\wt$ and  $T_{\bt}$) to the points $-1$, $1$, $\infty$. 
If we conjugate $g_5$  by this map 
  $\varphi$, then we obtain   the Latt\`{e}s map 
  $f_5=\varphi \circ g\circ \varphi^{-1}$  (see
  Example~\ref{ex:Lattes244} for more details). 
The homeomorphism $\varphi$ can be used in an obvious way to transfer  
$(\DD^0, \DD^1,  L)$ to  an isomorphic two-tile subdivision rule 
$(\widetilde{\DD}^1, \widetilde{\DD}^0, \widetilde{L})$ on
  $\CDach$. It is realized by $f_5$. 
  

  \begin{figure}
    \centering
    \begin{overpic}
      [width=12cm,tics=10,
      ]{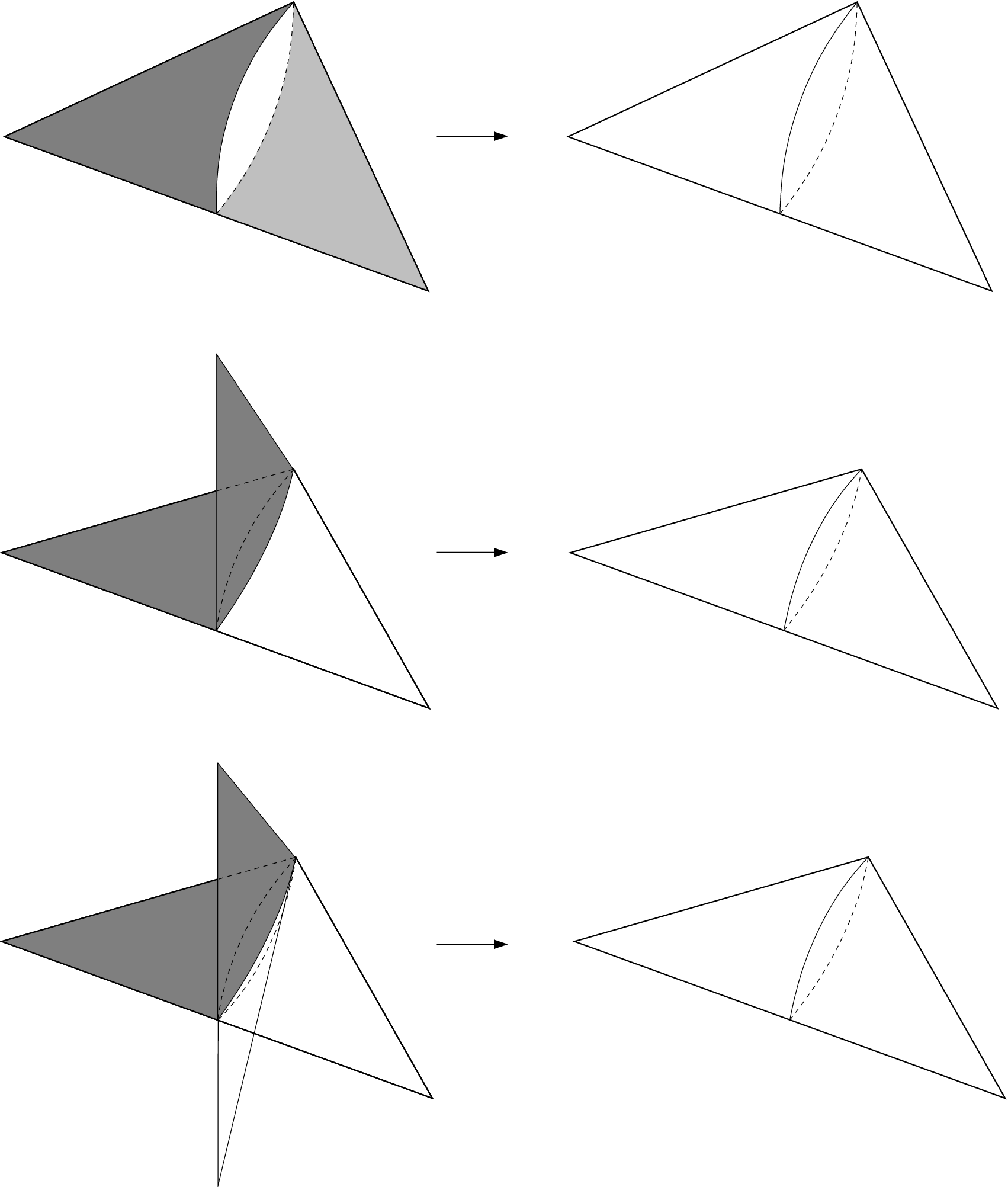}
      %
      \put(46,86){$\scriptstyle{-1}$}
      \put(71.5,100.2){$\scriptstyle{\infty}$}
      \put(83,73.5){$\scriptstyle{1}$}
      \put(65,75){${\Delta}$}
      %
      \put(47,51){$\scriptstyle{\omega}$}
      \put(71.5,61){$\scriptstyle{\infty}$}
      \put(83.5,38.5){$\scriptstyle{1}$}
      \put(65,40){${\Delta}$}
      %
      \put(47,18.5){$\scriptstyle{\omega}$}
      \put(72,28.5){$\scriptstyle{\infty}$}
      \put(84,5.5){$\scriptstyle{1}$}
      \put(65,8){${\Delta}$}
      %
      \put(-2,84.7){$\scriptstyle{-1\mapsto -1}$}
      \put(22.5,100.2){$\scriptstyle{\infty\mapsto 1}$}
      \put(16,79.4){$\scriptstyle{0\mapsto \infty}$}
      \put(34.5,73.5){$\scriptstyle{1\mapsto -1}$}
      \put(12,75){${\Delta}$}
      %
      \put(-2,50){$\scriptstyle{\omega \mapsto \omega}$}
      \put(18,70.4){$\scriptstyle{\mapsto \omega}$}
      \put(16,44.5){$\scriptstyle{0\mapsto \infty}$}
      \put(24.5,61){$\scriptstyle{\infty\mapsto 1}$}
      \put(34.5,38){$\scriptstyle{1\mapsto \omega}$}
      \put(12,40){${\Delta'}$}
      %
      \put(-3,18){$\scriptstyle{\omega \mapsto \omega}$}
      \put(18,35.9){$\scriptstyle{\mapsto \omega}$}
      \put(12.5,12.5){$\scriptstyle{0\mapsto \infty}$}
      \put(24.7,28.3){$\scriptstyle{\infty\mapsto 1}$}
      \put(34.5,5.7){$\scriptstyle{1\mapsto \omega}$}
      \put(19,0){$\scriptstyle{\mapsto \omega}$}
      \put(10,8){${\Delta'}$}
      \put(38.5,90){${g_5}$}
      \put(38.5,55){${\widetilde{g}_6}$}
      \put(38.5,22){${\widetilde{g}_7}$}
    \end{overpic}
    \caption{Adding flaps.}
    \label{fig:lattes_1_2_flaps}
  \end{figure}
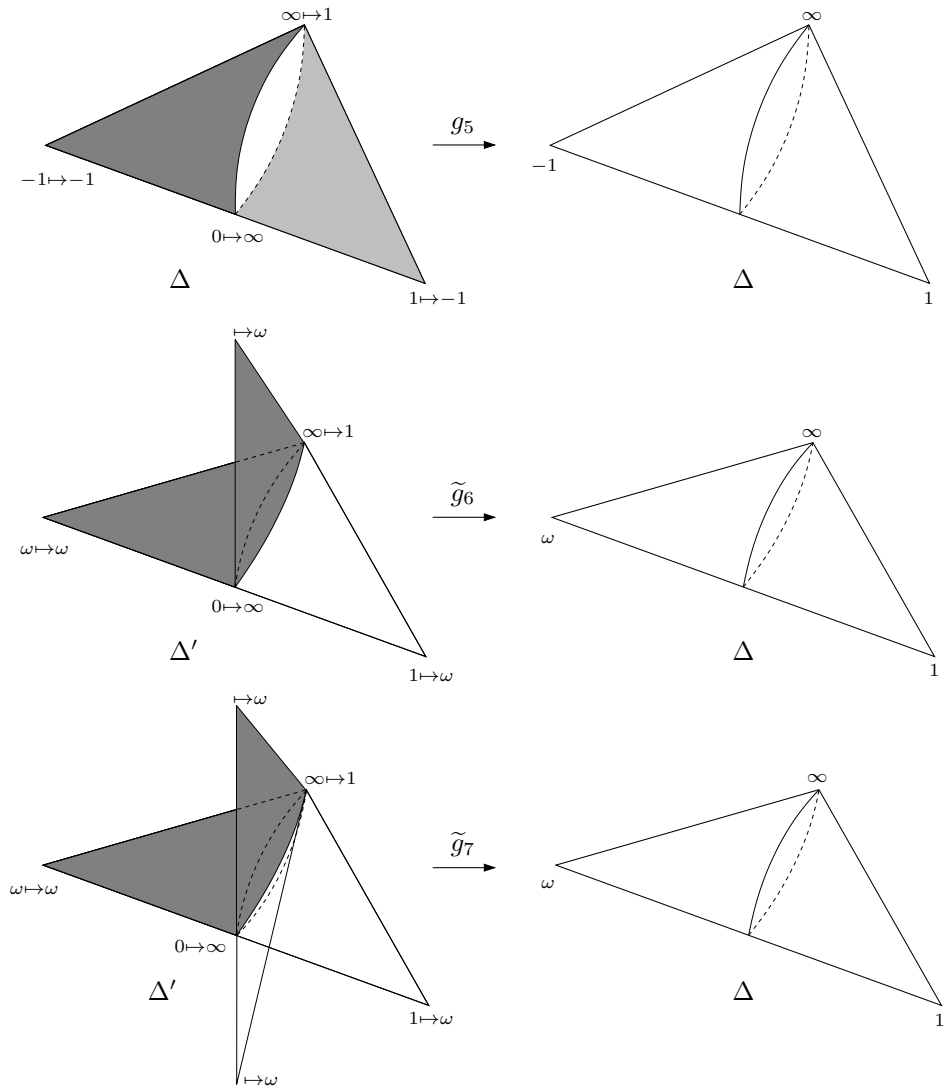
  
Similarly to Example~\ref{ex:obstructed_map}, we can modify  the map $g_5$ as follows.  Namely, we take the pillow $\Delta$ as above, but now label the vertices of
  $\Delta$ by  $\omega,1,\infty$ as shown on the
  middle right in  Figure~\ref{fig:lattes_1_2_flaps}. 
  Each side of $\Delta$ is subdivided into two triangles isometric to $T'$ as before.  We
   cut $T_\wt$ along the perpendicular bisector of its hypotenuse and glue in
  two isometric copies of $T'$. Informally, we refer to this
  procedure as ``adding a flap''.  This results in a surface
  $\Delta'$ homeomorphic to $\Delta$ that is built from six isometric copies $T_1,\dots,T_6$ of
  $T'$.  There is a map
  $\widetilde{g}_6\colon \Delta'\to \Delta$ that sends each $T_j$
  by a similarity to $T_\wt$ or $T_\bt$ as indicated in the
  picture. We can identify $\Delta'$ and $\Delta$ by a homeomorphism $\psi\: \Delta'\ra \Delta$  
  that respects the correspondence 
of the common three corners of
  $T_\wt$ and $T_\bt$ 
(labeled $\om, 1, \infty$ 
in both  $\Delta'$ and $\Delta$),  
  sends  the top of $\Delta'$ (consisting of four small triangles)
  to $T_\wt$, and the bottom of $\Delta'$ (consisting of two small triangles)
 to $T_\bt$. Then
  $g_6\coloneqq \widetilde{g}_6 \circ \psi^{-1}$ is a Thurston
  map with the three postcritical points $\om, 1, \infty$.  
  
This map was already considered in  
  Example~\ref{ex:tringflP}. There we saw that it is equivalent 
  to  the rational
  Thurston map 
  $f_6(z) = 1 + (\omega-1)/z^3$ with $\omega= e^{4\pi\iu/3}$. 
%

  It is possible to generalize this construction. For example, instead
  of adding just one flap on  the top face 
  $T_\wt$ of $\Delta$, we may add one flap on $T_\wt$ and
  $T_\bt$ each. This is illustrated at the bottom in 
  Figure~\ref{fig:lattes_1_2_flaps}.

  Moreover, instead of adding just one flap to the $1$-edge bisecting
  $T_\wt$, we can add $n\in \N_0$ flaps. Similarly, we can glue in $m\in \N_0$
  flaps at the $1$-edge bisecting $T_\bt$. This again results in
  a polyhedral surface $\Delta'$ consisting of $2n+2$ triangles
  isometric to $T'$ on the top, and $2m+2$ small triangles 
  isometric to $T'$ on  the bottom of $\Delta'$. 
  We label the vertices of $\Delta$ by
  $\omega,1,\infty$ as in the middle of
  Figure~\ref{fig:lattes_1_2_flaps}, and 
consider them as vertices 
of $\Delta'$ as well. There is a unique small triangle $\widetilde T$ with $\om \in \widetilde T$ that is contained in the top part of $\Delta'$. We color the tiles of $\Delta'$ in checkerboard fashion so that  $\widetilde T$ is black (as indicated in 
 Figure~\ref{fig:lattes_1_2_flaps}). 
 Then with a proper choice of an orientation on $\Delta'$ there  is a unique branched covering  map
  \begin{equation}\label{eq:wtg}
  \widetilde{g}\colon \Delta'\to \Delta
  \end{equation} 
   that sends each of the
  small triangles in $\Delta'$ to either $T_\wt$ or $T_\bt$ by a
  similarity, fixes the vertex $\omega$, and respects the coloring of tiles. 
  Note that $\widetilde{g}$ is not a Thurston map, because the domain $\Delta'$ and the  range $\Delta$ of $\widetilde{g}$ are different sets. To obtain a Thurston map,  
   we consider, as before, 
 an identification of $\Delta'$ and $\Delta$ by an 
 orientation-preserving   homeomorphism $\psi\colon \Delta' \ra
 \Delta$ 
that respects the correspondence 
of the points labeled $\omega,1,\infty$, sends the
  top of $\Delta'$ to
  $T_\wt$, and the  bottom   of $\Delta'$ to $T_\bt$.
Then  
\begin{equation} \label{eq:geng}
g\coloneqq \widetilde{g}\circ \psi^{-1}\colon \Delta\ra \Delta
\end{equation}
 is a Thurston map with
  $\post(g) =\{\omega, 1, \infty\}$. Moreover,  if $\CC$ is the common boundary of $T_\wt$  and $T_\bt$, then 
  $\CC$ is a $g$-invariant Jordan curve with 
 $\post(g)\sub \CC$. 

 Figure~\ref{fig:triag2flaps} illustrates  a special case of this construction. 
 It shows the subdivision rule realized by the Thurston map   $g=g_7$ 
 that arises according to Proposition~\ref{prop:ThmapSub} if we add one flap at the top face and one at the bottom face of 
 $\Delta$ (corresponding to  $n=m=1$).  
   In the figure   we have to identify   edges  as indicated 
   in order to obtain a
  $2$-sphere.

\ifthenelse{\boolean{nofigures}}{}{  
  \begin{figure}
    \centering
    \begin{overpic}
      [width=12cm, tics=10, 
      ]{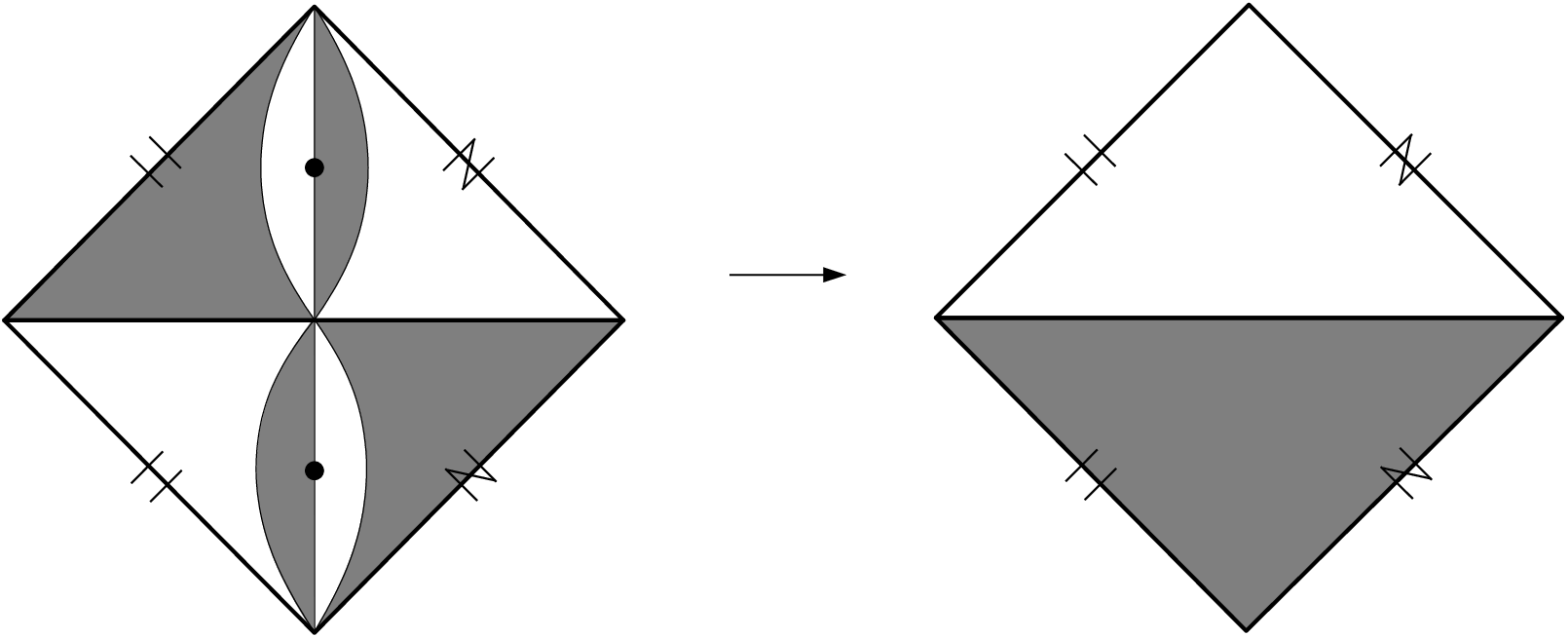}
      %
      \put(-3.5,16){$\scriptstyle{\omega\mapsto \omega}$}
      \put(22,21){$\scriptstyle{\mapsto \infty}$}
      \put(38,16){$\scriptstyle{1\mapsto \omega}$}
      \put(20,41){$\scriptstyle{\infty\mapsto 1}$}
      \put(19,30.7){$\scriptstyle{\mapsto \omega}$}
      \put(19,11.5){$\scriptstyle{\mapsto \omega}$}
      %
      \put(58.5,17.5){$\scriptstyle{\omega}$}
      \put(99,17.5){$\scriptstyle{1}$}
      \put(80,41){$\scriptstyle{\infty}$}
      \put(80,-1){$\scriptstyle{\infty}$}
      \put(48.5,25){${g_7}$}
    \end{overpic}
    \caption{Two-tile subdivision rule realized by $g_7$.}
    \label{fig:triag2flaps}
  \end{figure}
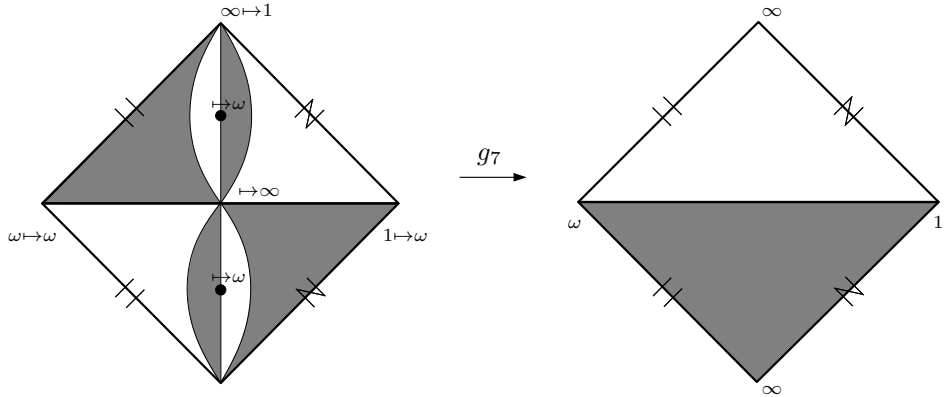
}

  Since the Thurston map $g$ in \eqref{eq:geng} has three postcritical points, it
  follows from Theorem~\ref{thm:3postrat} that $g$ is
  Thurston equivalent to a rational map $f\colon \CDach\to
  \CDach$. In fact, we may choose $f$ as
  \begin{equation}
    \label{eq:def_fmn_flaps}
    f(z) = 1+ \frac{\omega -1}{z^d}, 
  \end{equation}
  where $d=n+m+2$ and $\omega= e^{2\pi\iu (n+1)/d}$.
   The proof
  that $g$ is Thurston equivalent to $f$ is similar to the one
  given in Example~\ref{ex:tringflP};  we omit the details. 
   Note that the maps $f_5$, $f_6$, and  $f_7(z) = 1-
  2/z^4$ (which was considered in
  Example~\ref{ex:f2flaps}) are special cases of \eqref{eq:def_fmn_flaps}. 
  
  The map $f$ in  \eqref{eq:def_fmn_flaps} can be  obtained  more explicitly from
  $\widetilde{g}$ in \eqref{eq:wtg} as follows. Since $\Delta'$ and $\Delta$ are
  polyhedral surfaces, they are Riemann
  surfaces. By the  uniformization theorem there are
  conformal maps $\varphi\colon \Delta \to \CDach$ and
  $\widetilde \varphi\colon \Delta'\to \CDach$. We  can normalize them so that 
  $\widetilde \varphi(\omega) = \varphi(\omega)=\omega$,
  $\widetilde \varphi(1) = \varphi(1)=1$,
  $\widetilde\varphi(\infty) = \varphi(\infty)=\infty$. 
%
  Then one can show that
    $f= \varphi \circ \widetilde{g} \circ \widetilde  \varphi^{-1}$ is
  exactly the map given in \eqref{eq:def_fmn_flaps}.

  Let $\CC\subset \Delta$ be the $g$-invariant 
  Jordan curve as
  before. Then  $g$ is combinatorially
  expanding for $\CC$. It follows from  Theorem~\ref{thm:combexp1}
  that by possibly choosing a different identification 
  $\psi\: \Delta'\ra \Delta$ in the definition of $g$, we may assume that  $g$ is expanding. 
 Since $f$ is also expanding as follows from Proposition~\ref{prop:rationalexpch}, 
  the maps $g$ and $f$ are 
  topologically conjugate  by Theorem~\ref{thm:exppromequiv}.  So there is a homeomorphism
  $h\colon \Delta\to \CDach$ such that $f = h \circ g \circ
  h^{-1}$. Then  $\widetilde{\CC}\coloneqq h(\CC)\subset
  \CDach$ is an $f$-invariant Jordan curve with 
$\post(f) \subset
  \widetilde{\CC}$. This argument is closely related to
  the general construction of invariant Jordan curves in Chapter~\ref{cha:constructc}. In
  Figure~\ref{fig:invC_constr} the invariant Jordan curve 
  $\widetilde{\CC}$ is shown for the map $f_6$. 

  Let $(\DD^1,\DD^0,L)$ be the two-tile subdivision rule given by $g$
  and the $g$-invariant Jordan curve $\CC$ according to
  Proposition~\ref{prop:ThmapSub}, and
  $(\widetilde{\DD}_1, \widetilde{\DD}_0, \widetilde{L})$ be the
  one  
  given by $f$ and
  the $f$-invariant Jordan curve $\widetilde{\CC}$. 
  Then these 
  two-tile subdivision rules are isomorphic (the isomorphism is
  naturally induced by $h$). In this sense, $g$ and $f$ realize
  the same two-tile subdivision rule.  
\end{ex}

\begin{ex}[Subdivision rules from tilings]
  \label{ex:maps_from_tilings} 
We now describe a general method for  obtaining subdivisions
rules from tilings of the Euclidean or hyperbolic plane. This can
be used to find Thurston maps 
with arbitrarily large sets 
of postcritical points. 

First, we consider the unit square $[0,1]^2$ in $\R^2$ and its translates under the lattice 
$\Z^2\sub \R^2$. These squares form a tiling of $\R^2$. More precisely, they are 
 $2$-dimensional cells or 
tiles  of a cell decomposition $\DD$ of $\R^2$ whose vertex set is $\Z^2$ and  whose  $1$-skeleton is the ``square grid'' $S=(\Z\times\R)\cup (\R\times \Z)$.
Let $\CC$ be a  Jordan
  curve contained in $S$, and  $X\subset \R^2$ be the closed
  Jordan region  with boundary $\CC$. Then $X$ is a union of tiles in $\DD$.
  We assume $X$ consists of at least two such tiles. 
 We  take  two identical
  copies $X_\wt$ and $X_\bt$ of $X$ (which we call $0$-tiles),
  and glue  them together along their boundaries to form a pillow
  $\Delta$. Note that this common boundary can be identified with $\CC$.
   Among the vertices of $\DD$ (i.e., the lattice points $\Z^2$) contained in $\CC\sub \Delta$  we fix
  four distinct ones and  consider them as the $0$-vertices of $\Delta$.  The four arcs 
    into which they divide  $\CC$ are the  $0$-edges. These $0$-tiles, $0$-edges,
  and $0$-vertices 
  form a cell decomposition $\DD^0$ of
  $\Delta\cong S^2$.

 Our tiling of $\R^2$ by copies of $[0,1]^2$ gives a natural subdivision of 
 $X$, and hence also of  $X_\wt$ and $X_\bt$,  into unit squares. Let $\DD^1$ be the cell
  decomposition of
  $\Delta\cong S^2$ that is given by these squares as $1$-tiles, their sides as
   $1$-edges, and their corners 
  as $1$-vertices. Clearly,  $\DD^1$ is a refinement of $\DD^0$ and
  every tile in $\DD^1$ is a $4$-gon.  Every $1$-vertex
  $v\notin \CC$ is contained in four $1$-tiles, and every
  $1$-vertex $v\in \CC$ is contained in the same number of  
   $1$-tiles in $X_\wt$ and in $X_\bt$. 
  It follows that every $1$-vertex is contained in an even number of
  $1$-tiles. We conclude  that the pair  $(\DD^1,\DD^0)$ satisfies the
  conditions \ref{item:subdivcomb1}--\ref{item:subdivcomb4} in
  Definition~\ref{def:subdivcomb}.

 To define a corresponding labeling,  we  fix a $1$-tile $X'$, a $1$-vertex $v'\in X'$, and 
  a $0$-vertex $v$. Since $v\in X_\wt$, there is a unique
orientation-preserving   labeling $L\colon \DD^1\to \DD^0$ such that $L(X')= X_\wt$ and
  $L(v')=v$ by
  Lemma~\ref{lem:labeluniq}~\ref{item:labeluniq2}. 
  Then $(\DD^1, \DD^0, L)$ is a two-tile subdivision rule. By Proposition~\ref{prop:rulemapex}
   it can be realized by a Thurston map 
$f\colon \Delta\to \Delta $. Roughly speaking, $f$ is
  constructed by mapping $X'$ to $X_\wt$, normalized such that
  $f(v')=v$, 
and by extending to all 
of $\Delta$ ``by reflection''. 

  Note that the Latt\`{e}s map from Section~\ref{sec:Lattes} and
  the Latt\`{e}s-type map from Example~\ref{ex:2x3} may be viewed
  as examples of this procedure. In \cite{HP_Sierp} Ha\"{i}ssinsky and Pilgrim constructed
   certain rational maps with Sierpi\'{n}ski
  carpet Julia sets in this way.

  Instead of square tilings,  
one can also use other tilings of $\R^2$ 
  for this construction.  The  Latt\`{e}s maps in
  Examples~\ref{ex:Lattes244},~\ref{ex:Lattes333},~\ref{ex:lattes236}, 
  and the map in 
  Example~\ref{ex:barycentric} are of this form. See also
  \cite{Me02} and \cite{HM}. 

  Finally, we can use 
tilings of  the hyperbolic plane
  instead. For example, if $n\ge 5$ is fixed, then one can  tile the hyperbolic plane $\mathbb{H}^2$ with
   right-angled $n$-gons.  This gives a cell decomposition
    $\DD$ of  $\mathbb{H}^2$  so that  four  $n$-gons intersect
  at each vertex. 
Again we consider a (hyperbolic) pillow 
$\Delta$ obtained by gluing together 
 two copies of a  closed Jordan region $X\sub \mathbb{H}^2$ whose boundary is contained in the
  $1$-skeleton of $\DD$ and encloses at least two  $n$-gons of the tiling. 
   If we define a Thurston map $f\: \Delta \ra \Delta$ 
  by the analog of the above construction, then  (under some mild additional assumptions) $f$ will have  $n$
  postcritical points.  
\end{ex}

More examples of maps constructed from subdivision rules can be found
in \cite{CFKP}, and more examples of subdivisions in
\cite{CFP06b}.    


Figures \ref{fig:tiles8_z2-1}, \ref{fig:tiles_bary}, and
\ref{fig:invC_constr} show \emph{symmetric conformal
  tilings} of $\CDach$. This means that if two tiles share an edge, then  they are
conformal reflections of each other along this edge. The  
tiling can be produced by successive reflections, and so each
individual tile encodes the information for the whole tiling
(more on this subject can be found in \cite{BowSt17}).

\ifthenelse{\boolean{singlechapter}}{ 

%


\chapter{Quotients of Thurston maps}
\label{cha:quotiens}

In this chapter we study the general problem when a Thurston map $f\: S^2\ra S^2$ passes to another Thurston map on a
quotient of $S^2$ induced by an equivalence relation $\Sim$
(see 
Section~\ref{sec:appquotmaps} for some basic facts about equivalence relations and quotient spaces).
Here  the quotient space $\widetilde S^2 \coloneqq S^2/\Sim$
equipped with the quotient topology has to be a
$2$-sphere itself. Well-known sufficient conditions for this  to be the case are due to Moore. We
call an equivalence relation $\sim$ on $S^2$ that satisfies these
conditions to be of {\em Moore-type} (see
Definition~\ref{def:eq_Moore_type}). So if $\sim $ is of
Moore-type, then the quotient space $\widetilde S^2$ is a
$2$-sphere (see Theorem~\ref{thm:moore}).
 
 We denote by  $[x]\coloneqq \{ y\in S^2: y\sim x\}$  the equivalence class of a point  $x\in S^2$, and by $\pi\: S^2 \ra \widetilde S^2=S^2/\Sim$ the quotient map defined as 
$\pi(x)=[x]\in S^2/\Sim$ for $x\in S^2$. The general question
when $f\: S^2\ra S^2$ 
{\em descends}\index{map!descending to quotient}\index{quotient!map descending to}\index{descending!to quotient} 
to the quotient
$\widetilde{S}^2$, 
i.e., when there exists a map
$\widetilde{f}\colon \widetilde{S}^2 \to \widetilde{S}^2$ such
that  $\pi\circ f = \widetilde{f} \circ \pi$, is very easy to
answer. Namely,  $\sim$ needs to be  
{\em $f$-invariant}
in the sense that we have the implication 
\index{equivalence relation!f-invariant@$f$-invariant}\index{f-invariant@$f$-invariant!equivalence relation}\index{invariant!equivalence relation}
\begin{equation}
  \label{eq:eq_f_inv}
  x\sim y \Rightarrow f(x) \sim f(y)
\end{equation}
for all $x,y\in S^2$  (see Lemma~\ref{lem:f_descends}).
This is equivalent to the
requirement that \begin{equation}
  \label{eq:eq_f_inv2}
  f([x]) \subset [f(x)]
\end{equation}
for each $x\in S^2$. 
 
If $\sim$ is $f$-invariant and the map $f\:S^2\ra S^2$ descends to a map
$\widetilde{f}\colon \widetilde{S}^2 \to \widetilde{S}^2$, then we have  the following
commutative diagram: 
\begin{equation}
  \label{eq:sim_descends}
  \xymatrix{
    S^2 \ar[r]^f \ar[d]_\pi & S^2 \ar[d]^\pi
    \\
    \widetilde{S}^2 \ar[r]^{\widetilde{f}} & \widetilde{S}^2\rlap{.}
  }
\end{equation}
We call   $\widetilde f$ the {\em quotient map} of $f$ on $ \widetilde{S}^2$. It is uniquely determined (by $\sim $ and $f$)  and continuous.

Even if  $f$ is a Thurston map and $\sim$ is an $f$-invariant 
equivalence relation of Moore-type, it is not guaranteed that the map 
$\widetilde f$ as in \eqref{eq:sim_descends} defined on  the $2$-sphere $\widetilde S^2$ 
is a Thurston map or even a branched covering map
 (see Example~\ref{ex:ftilde_not_Thurston}).
 For this we need a stronger condition on $\sim$. 
\begin{definition}[Strongly invariant equivalence relations]
  \label{def:sim_strongly-inv}
  \index{equivalence relation!f-invariant@$f$-invariant!strongly}
  \index{strongly f-invariant@strongly $f$-invariant equivalence relation}
\index{f-invariant@$f$-invariant!equivalence relation!strongly}\index{invariant!equivalence relation!strongly}
  Let $f\colon S^2\to S^2$ be a branched covering map. Then an 
  equivalence 
  relation $\sim$ on $S^2$ is called \emph{strongly
    $f$-invariant} if the image of each equivalence class is an
  equivalence 
    class, or equivalently, if 
    \begin{equation*}
      f([x]) = [f(x)]
    \end{equation*}
    for each $x\in S^2$.
\end{definition}
Clearly, each  strongly $f$-invariant equivalence relation is
$f$-invariant. 
Based on this concept,
we can  state the main result of this chapter. 
\begin{theorem}[Quotients of branched covering maps]
  \label{thm:f_descends_branched_cover}
  \index{branched covering map!descending to quotient}\index{quotient!of branched covering map}\index{descending!to quotient}
  Suppose   $f\colon S^2\to S^2$ is  a branched covering map,  $\sim$
is    an 
  $f$-invariant equivalence relation of Moore-type on
  $S^2$, $\pi \: S^2\ra  \widetilde{S}^2\coloneqq S^2/\Sim$ is  the quotient map,
  and $ \widetilde{f}\: \widetilde{S}^2 \ra  \widetilde{S}^2$
 is the induced map as in \eqref{eq:sim_descends}. Then 
    $\widetilde{f}$ is a branched covering map if and only
      if  $\sim$  is strongly $f$-invariant.
      
 Moreover, in this case  the following statements are true: 
  \begin{enumerate}
  \item 
    \label{item:f_descends_deg}
    $\deg(\widetilde{f}) = \deg(f)$.
  \item 
    \label{item:f_descends_locdeg}
    $     \crit(\widetilde{f})=  \pi(\crit(f))$ and $ \post(\widetilde{f}) =\pi(\post(f))$.
     \item 
    \label{item:f_descends_critpre}
     
    If $x\in S^2$ and the equivalence class $[x]$ contains the (distinct) 
    critical points $c_1,\dots, c_n\in S^2$ of $f$, then the
    local degree of $\widetilde{f}$ at $[x]\in \widetilde S^2$ is given by
    \begin{equation}
      \label{eq:loc_deg_ftilde}
      \deg(\widetilde{f}, [x])= 1 + \sum_{i=1}^n (\deg(f,c_i) -1). 
    \end{equation}
  \end{enumerate}
\end{theorem}

We obtain the following immediate consequence.
\begin{cor}[Quotients of Thurston maps]
  \label{cor:f_descends_Thurston}
  \index{Thurston map!descending to quotient}\index{quotient!of Thurston map}\index{descending!to quotient}
  Suppose  $f\colon S^2\to S^2$ is  a Thurston map,  $\sim$ is  a strongly 
  $f$-invariant equivalence relation of Moore-type on
  $S^2$, and  $ \widetilde{f}\: \widetilde{S}^2 \ra  \widetilde{S}^2$
 the induced map on the quotient 
  $\widetilde{S}^2=S^2/\Sim$ as in \eqref{eq:sim_descends}. Then  $\widetilde{f}$ is a Thurston map. 
 
 Moreover,  $\deg(\widetilde{f}) = \deg(f)$, and, if $\pi \: S^2\ra  \widetilde{S}^2= S^2/\Sim$ is  the quotient map,  $\crit(\widetilde{f}) = \pi(\crit(f))$, and $\post(\widetilde{f}) = \pi(\post(f))$.  
\end{cor}

\begin{proof} 
  It follows from Theorem~\ref{thm:f_descends_branched_cover}
  that $\widetilde{f}$ is a branched covering map on the
  $2$-sphere $\widetilde S^2$ with the properties specified in
  the second part of the statement. Since $f$ is a Thurston map, we
  have  
$$\#\post(\widetilde{f})=\#\pi(\post(f)) \le \#\post(f)<\infty$$ and 
$\deg (\widetilde{f})=\deg(f)\ge 2$. Hence $\widetilde f$ is also a Thurston map.
\end{proof}

 Quotients of rational maps (not necessarily
postcritically-finite rational maps)   were  considered by McMullen 
 in a somewhat different setting  (see \cite[Appendix B]{McM}).

This chapter is organized as follows. In Section~\ref{sec:clos-equiv-relat} we review some facts about
equivalence relations relevant for the statement of Moore's theorem.
 In particular, we discuss the important concept of a {\em closed equivalence relation}
(see Definition~\ref{def:closed_eq}) and state various conditions that characterize 
closed equivalence relations. 

In Section~\ref{sec:toppropbrcov} we   prove two facts about branched covering maps that are relevant for the proof of Theorem~\ref{thm:f_descends_branched_cover} (see Lemma~\ref{lem:pre_conn_comp} and Lemma~\ref{lem:Jordan_comp}). 

Finally, in Section~\ref{sec:quot-thurst-maps} we discuss some properties of strongly invariant equivalence relations (see Lemma~\ref{lem:eq_strong_inv}) and establish 
a fact needed in the proof of Theorem~\ref{thm:f_descends_branched_cover} (see Lemma~\ref{lem:sim_str_inv_map_eqclass}).  The proof of this theorem concludes  the section and the chapter.

\section{Closed equivalence relations and Moore's theorem}
\label{sec:clos-equiv-relat}

We first take a closer look at the situation when a topological space $\widetilde{X}$ is obtained
as the quotient of some other topological space ${X}$
by an  equivalence relation. Often $\widetilde{X}$ is then called
a \emph{decomposition space} (for a standard reference see \cite{Da}).  

Let $\sim$ be an equivalence relation on a set
${X}$. As before,  we denote by $[x]$ the equivalence
class of a point $x\in{X}$, by ${X}/\Sim \coloneqq
\{[x] : x\in {X}\}$ the quotient space,  and by
$\pi\colon {X} \to{X}/\Sim$ the 
quotient map given by $\pi(x) = [x]$ for $x\in
{X}$. If $X$ is a topological space, then we equip ${X}/\Sim$ with the quotient topology. 
See Section~\ref{sec:appquotmaps} for more details. 

A set $A\sub X$ is called 
\emph{saturated}\index{saturated!set}
if $x\in A$ and $x\sim y$ imply $y\in A$ for all $x,y\in X$, or equivalently, 
if $A$ is a union of equivalence classes. 

We define the 
\emph{saturated interior}\index{saturated!interior}
of a set
$U\subset{X}$ as
\begin{align}
  \label{eq:def_sat_int}
  U_s \coloneqq&\; \bigcup\{[x] : x\in X, \, [x]\subset U\}
       \\ \notag
     =&\;\bigcup\{A\subset {X} : 
       A \text{ is saturated and } A\subset U\}. 
\end{align}
The saturated interior of $U$ is the largest saturated set
contained in $U$. 

\begin{lemma}[Closed equivalence relations]
  \label{lem:closed_eq}
  Let $X$ be a compact metric space and $\sim$ be an equivalence
  relation on $X$. Then the following conditions are equivalent:
  \begin{enumerate}
  \item 
    \label{item:closed_eq1}
    The set $\{(x,y) : x,y\in X,\, x\sim y \} \subset X\times X$ is
    closed.
  \item 
    \label{item:closed_eq2}
    Let $\{x_n\}_{n\in \N}$ and  $\{y_n\}_{n\in \N}$ be convergent
    sequences in $X$. Then
    \begin{equation*}
      x_n\sim y_n
      \text{ for all } n\in \N  \text{ implies } \lim_{n\to \infty}  x_n\sim \lim_{n\to \infty} y_n. 
    \end{equation*}
  \item 
    \label{item:def_usc}
    For each $x\in X$  and each  neighborhood $U\sub X$ of
    $[x]$ there 
    is a neighborhood $V\subset U$ of $[x]$ such that 
    \begin{equation*}
      [y]\cap V\neq \emptyset \Rightarrow [y] \subset U \text{ for all $y\in X$}.
    \end{equation*}
  \item 
    \label{item:def_closed_sat_int}
    For each open set $U\subset X$ the saturated interior
    $U_s$ is open. 
  \end{enumerate}
\end{lemma}  

In the first condition $X\times X$ is equipped with the product
topology. Each of  these four equivalent conditions implies
that equivalence classes  are closed, and hence compact subsets
 of $X$.

\begin{definition}[Closed equivalence relations]
  \label{def:closed_eq}
  \index{equivalence relation!closed}
  \index{closed equivalence relation}
  An equivalence relation $\sim$ on a compact metric space $X$
  is called \emph{closed} if it satisfies one of the conditions (and hence every condition) 
  in Lemma~\ref{lem:closed_eq}.  
\end{definition}

\index{equivalence relation!upper semi-continuous}
\index{upper semi-continuous equivalence relation}
In the literature closed equivalence relations are often called
``upper semi-continuous'' instead. One may define upper
semi-continuous equivalence relations in any topological space
(usually by condition \ref{item:def_usc} in 
Lemma~\ref{lem:closed_eq}, together with the requirement that
each equivalence class is compact). To simplify the discussion,
we chose to restrict ourselves to compact metric spaces.

 In \cite{MeyPet} several
other   characterizations of closed equi\-valence relations  can be found.

\begin{proof}[Proof of Lemma~\ref{lem:closed_eq}]

  The equivalence of \ref{item:closed_eq1} and
  \ref{item:closed_eq2}  immediately follows from  the definition of the
  product topology. 

   \smallskip
  \ref{item:closed_eq2} $\Rightarrow$ \ref{item:def_usc}
   We first note that \ref{item:closed_eq2} implies that each
  equivalence class $[x]$, $x\in X$,  is closed, and hence compact. Indeed, if we 
  choose $x_n=x\in X$ to be a constant sequence,  then
  \ref{item:closed_eq2} shows that the limit of any convergent
  sequence $\{y_n\}$ in  $[x]$ is contained in $[x]$.

 We now argue by contradiction and  assume
  that \ref{item:closed_eq2} is satisfied, but
  \ref{item:def_usc} is not. Let $[x]$ be an equivalence class
  and $U$ be a neighborhood of $[x]$ that violates
  \ref{item:def_usc}. Since $[x]$ is compact, it follows that
  for each  sufficiently large $n\in \N$ the $1/n$-neighborhood of
  $[x]$ with respect to the underlying metric on $X$ satisfies 
   $\mathcal{N}_{1/n}([x]) \subset U$. By
  assumption, for each large enough $n\in \N$ there exists $x_n \in \mathcal{N}_{1/n}([x])$ such
  that $[x_n] \not\subset U$. This means
  there exists  $y_n \in [x_n]$ with $y_n\notin U$. 
Taking subsequences, 
we may  assume that $\{x_n\}$ and $\{y_n\}$ are
  convergent. Then  $y \coloneqq \lim_{n\to \infty}  y_n$ is contained in the closure
  of $X\setminus U$, and so  $y\notin [x]$. On the other
  hand,  $x'\coloneqq \lim_{n\to \infty} x_n \in [x]$,  and so  $x'\not \sim y$.
   This is a contradiction to  \ref{item:closed_eq2},  proving  the
  claim.

  \smallskip
  \ref{item:def_usc} $\Rightarrow$ \ref{item:def_closed_sat_int}
  Let $U\subset X$ be an open set, and $x\in U_s$ be arbitrary. Then $[x]\sub U$,
  and so   we can find a neighborhood  $V\subset U$
   of $[x]$ that satisfies the condition in 
  \ref{item:def_usc}. This condition implies that  $V\subset U_s$, and so $U_s$ 
  is a neighborhood of  $x$. 
Since $x\in U_s$ 
  was
  arbitrary, it follows that $U_s$ is open.

  \smallskip
  \ref{item:def_closed_sat_int} $\Rightarrow$
  \ref{item:closed_eq2} 
 We  first note that
  \ref{item:def_closed_sat_int} implies that every equivalence
  class  is closed. Indeed, let $x\in X$ be arbitrary and 
  consider the open set $U=X\setminus\{x\}$. Then $U_s=U\setminus[x]$.
  By \ref{item:def_closed_sat_int} the set $U_s$ is open, and so 
  $[x]=X\setminus U_s$ is closed.

Now let    $\{x_n\}$ and  $\{y_n\}$ be convergent sequences in $X$ 
  with $x_n \sim y_n$ for all $n\in
  \N$. Define     $x=\lim_{n\to \infty}  x_n $ and $y=\lim_{n\to \infty}  y_n$. 
 For  $\eps>0$ let $U=\mathcal{N}_\epsilon([x])$. Then by 
 \ref{item:def_closed_sat_int} the set $U_s$ is open, and it contains $x\in [x]\sub U_s$. Hence 
 $x_n\in U_s$ and so $y_n\in U_s$ for sufficiently large $n$. This implies that 
 $y\in \overline{U}_{\!s} \sub \overline{U}$, and so $\dist(y, [x])\le \eps$. Since 
 $\eps>0$ was arbitrary and $[x]$ is closed, we have $y\in [x]$ and so $x\sim y$ as desired. 
  \end{proof}

An equivalence relation $\sim$ on a compact metric space $X$ is
called 
\emph{monotone},\index{monotone equivalence relation}\index{equivalence relation!monotone}
if every equivalence class  of $\sim$ is connected. The following statement
is well known (see \cite[Proposition~1.4.1, p.~18]{Da}).  We
provide a proof for the convenience of the reader. 

\begin{lemma}
  \label{lem:pre_monotone}
  Let $\sim$ be a closed and monotone equivalence relation on a 
  compact metric space $X$, let $\pi \: X\ra \widetilde{X}\coloneqq  X/\Sim$ be  the quotient map, and 
   $\widetilde{K}\subset \widetilde{X}$ be a
  connected set. Then $K \coloneqq \pi^{-1}(\widetilde{K})
  \subset X$ is connected.  
\end{lemma}

Recall that  the quotient space $ \widetilde{X}= X/\Sim$ is equipped
with the quotient topology. 

\begin{proof}
  Assume the statement is not true. Then there are disjoint open 
  sets $U,V\subset X$ such that $K_U\coloneqq K\cap U$ and
  $K_V\coloneqq K \cap V$ are non-empty and $K\subset U\cup
  V$. Each equivalence class $[x] \subset K=\pi^{-1}(\widetilde{K})$
  is either contained in $K_U$ or in $K_V$, because $[x]$ is connected by our hypotheses. This implies that 
  $K_U$ and $K_V$ are saturated sets.

  Let $U_s$ and $V_s$ be the saturated interiors of $U$ and $V$, 
  respectively.  Then $U_s$ and $V_s$ are disjoint, 
  $K_U\sub U_s$ and $K_V\sub V_s$. Since $\sim$ is closed, $U_s$ and $V_s$  are open sets by condition~\ref{item:def_closed_sat_int} in 
  Lemma~\ref{lem:closed_eq}. Thus
  $\widetilde{U}_s \coloneqq  \pi(U_s)$ and
  $\widetilde{V}_s\coloneqq  \pi(V_s)$ are disjoint 
  open  sets in $\widetilde{X}$. Moreover, 
       $\widetilde{U}_s \cap \widetilde{K}  =\pi(K_U) \neq \emptyset$, 
       $ \widetilde{V}_s \cap \widetilde{K} =\pi(K_V) \neq \emptyset$, and 
      $\widetilde{K}\sub 
     \widetilde{U}_s\cup \widetilde{V}_s$.
  This contradicts our assumption that $\widetilde{K}$ is
  connected. \end{proof}

After these general considerations, we now turn to equivalence relations on a $2$-sphere $S^2$. 

\begin{definition}[Moore-type equivalence relations]
  \label{def:eq_Moore_type}
  \index{equivalence relation!Moore-type}
  \index{Moore-type equivalence relation}
  An equivalence relation $\sim$ on $S^2$ is said to be of
  \emph{Moore-type}, if the following conditions are satisfied:
  \begin{enumerate}
  \item 
    \label{item:Moore1}
    The  equivalence relation $\sim$ is closed.
  \item 
    \label{item:Moore2} 
     The  equivalence relation $\sim$ is monotone.  
  \item
    \label{item:Moore3}   
    No equivalence class of $\sim$ {\em separates} $S^2$, i.e., 
  $S^2\setminus [x]$ is connected for each $x\in S^2$.
  \item 
    \label{item:Moore4}  
    The equivalence relation $\sim$ is {\em non-trivial}, i.e., 
    there are at least two distinct equivalence classes.
  \end{enumerate}
\end{definition}

The reason for our terminology is the following important
theorem due to Moore. See \cite{Moo} for the
original proof, \cite[Theorem~25.1, p.~187]{Da} for a stronger
statement, and \cite[Supplement~1]{Ca} for a general discussion
on the $2$-sphere recognition problem.  

\begin{theorem}[Moore] 
  \label{thm:moore}
  \index{Moore's theorem}  
  Let $\sim$ be an equivalence relation on $S^2$ of
  Moore-type. Then the quotient space $S^2/\Sim$ is homeomorphic
  to $S^2$. 
\end{theorem}

As before, it is understood that   $S^2/\Sim$ is equipped with the
quotient topology. 
There is an equivalent description of 
Moore-type equivalence relations. 
To discuss this, we first require a definition. 

\begin{definition}[Pseudo-isotopies]
  \label{def:pseudo-isot}
  \index{pseudo-isotopy}\index{isotopy!pseudo}
  Let ${X}$ and ${Y}$ be topological spaces. A
  homotopy $H\colon {X}\times [0,1] \to {Y}$ is a
  \emph{pseudo-isotopy} of ${X}$ to ${Y}$ if for
  each $t\in [0,1)$ the map 
  $H(\cdot, t)\colon {X}\to {Y}$ is a
  homeomorphism.  
\end{definition}
So a pseudo-isotopy can fail to be a homeomorphism only at time
$t=1$.
An equivalence relation $\sim$ on a topological space
${X}$ is \emph{realized} or 
\emph{induced by a pseudo-isotopy}\index{equivalence relation!induced by!pseudo-isotopy}\index{pseudo-isotopy!equiv.\ relation induced by}
 if there is a pseudo-isotopy $H\colon {X}
\times [0,1] \to {X}$ with $H_0= H(\cdot, 0) =
\id_{{X}}$ such that
\begin{equation}
  \label{eq:def_eq_end_pseudo}
  x\sim y \Leftrightarrow H(x,1) = H(y,1)
\end{equation}
for all $x, y\in{X}$. 

Conversely, if a pseudo-isotopy $H\colon {X} \times
[0,1] \to {X}$ with $H_0= \id_{{X}}$ is given,
then \eqref{eq:def_eq_end_pseudo} defines an equivalence relation  $\sim$ on ${X}$ induced by $H$.

\begin{theorem}
  \label{thm:Moore_pseudo}
  Let $\sim$ be an equivalence relation on $S^2$. Then $\sim$ is
  of Moore-type if and only if $\sim$ is induced by a
  pseudo-isotopy 
  $H\colon S^2\times [0,1] \to S^2$. 
\end{theorem}

The ``if''-direction is the easy implication in this statement. Its proof  can be found in
\cite[Lemma~2.4]{Me14}; the proof of the 
``only if''-implication 
is quite
involved and can be found in \cite[Theorem~25.1 and
Theorem~13.4]{Da}.

 \begin{cor}\label{cor:psisoconq} Let  $\sim$ be  an 
  equivalence relation of Moore-type on
  $S^2$,  and  $\pi \: S^2\ra  \widetilde{S}^2\coloneqq S^2/\Sim$
  be the quotient map. 
Then the induced map on singular homology groups $\pi_*\: H_2(S^2)\ra H_2(\widetilde S^2)$ is an isomorphism.
\end{cor} 
   
  Recall that we always assume that $S^2$ is oriented. The given orientation 
  on $S^2$ can be represented  by a generator $[S^2]$ of $H_2( S^2)\cong \Z$, called the fundamental class of $S^2$  (see Section~\ref{sec:orient}). By the corollary we may choose an orientation on the $2$-sphere $\widetilde S^2$ such that  
   $\pi_*([S^2])$ is the fundamental class on $\widetilde S^2$. With these choices we then have $\deg(\pi)=1$ for the degree of $\pi$ (as defined in Section~\ref{sec:orient} in terms of the induced map on homology).
   
   \begin{proof} We know by Theorem~\ref{thm:Moore_pseudo} that $\sim$ is induced by a pseudo-isotopy 
   $H\: S^2\times [0,1]\ra S^2$. Let $h\coloneqq H_1$ be the time-$1$ map. Then $h$ is a map homotopic to 
   $H_0=\id_{S^2}$. Since the degree of a map is a homotopy invariant (see \cite[p.~134]{Ha}), we conclude that 
   $  \deg(h)=\deg(\id_{S^2})=1$. In particular, $h$ is surjective (it is a standard fact that a non-surjective continuous map on $S^2$ is null-homotopic and so has vanishing degree; see 
   \cite[p.~134]{Ha}). 
   
Since $\sim$ is induced by $H$, we know that $x\sim y$ if and
only if $h(x)=h(y)$ 
for all $x,y\in S^2$. 
This allows us to define a map $\varphi\: \widetilde S^2\ra S^2$ as follows.
If $x\in S^2$ we set 
$\varphi([x])\coloneqq    h(x)$. This map is well-defined, and a homeomorphism of  $\widetilde S^2=S^2/\Sim$ onto $S^2$ (see  
Lemma~\ref{lem:XsimY}~\ref{item:XsimY_cpt_T2}). 

Note that $h=\varphi \circ \pi$. Since $\varphi$ is a homeomorphism, we can choose a fundamental class $[\widetilde S^2]$ on the $2$-sphere 
$\widetilde S^2$, i.e., a  generator $[\widetilde S^2]$ of $H_2(\widetilde S^2)\cong \Z$, such that $\varphi_*([\widetilde S^2])=[S^2]$.  With this choice $\deg(\varphi)=1$ and so
$$\deg(\pi) =\deg(\varphi)\cdot \deg(\pi)=\deg(\varphi\circ \pi)=\deg(h)=1.$$
This implies that $\pi_*([S^2])=[\widetilde S^2]$ and so $\pi_*$ is an isomorphism. 
  \end{proof} 
  
  In the previous proof we saw explicitly how to find a homeomorphism $\varphi$ between the quotient space $\widetilde S^2=S^2/\Sim$ and the $2$-sphere $S^2$ if the equivalence relation $\sim$ is induced by a pseudo-isotopy. This shows that  
Theorem~\ref{thm:Moore_pseudo} implies Theorem~\ref{thm:moore}, and so 
 Theorem~\ref{thm:Moore_pseudo} can be regarded as a stronger version of Moore's theorem.

In Chapter~\ref{cha:combexp} we will also require  a $1$-di\-men\-sional version of Moore's
theorem. It can easily be derived from the topological
characterization of arcs and topological circles (in equivalent
form this is stated in \cite[Section~9.1, (1.1), p.~165] {Why} or as two exercises in \cite[Exercise~4.2 and
4.3, p.~21]{Da}).

\begin{prop}
  \label{lem:1dimMoore} 
  Let $J$ be an arc or a topological circle, and $\sim$ be an
  equivalence relation on $J$.  Suppose that
  \begin{enumerate}
   
  \item
    each equivalence class of $\sim$ is a compact and  connected
    subset of $J$,  
  
  \item
    there are at least two distinct equivalence classes. 
\end{enumerate}
Then the quotient space $\widetilde J=J/\Sim$ is an arc or a
topological circle, respectively.
\end{prop}

\section{Branched covering maps and continua}
\label{sec:toppropbrcov}

In this section we 
establish two
 topological properties of branched covering maps formulated in
 Lemma~\ref{lem:pre_conn_comp} and
 Lemma~\ref{lem:Jordan_comp}. They are needed for  the proof of
 Theorem~\ref{thm:f_descends_branched_cover}.\index{branched covering map!and continua} 

\begin{lemma}
  \label{lem:pre_conn_comp}
  Let $X$ and $Y$ be compact metric spaces,  
$f\colon X \ra Y$ be an open and continuous map,  and
  $K\subset Y$ 
  be a compact connected set. Then each component $C$ of
  $f^{-1}(K)$ satisfies $f(C) = K$.
 \end{lemma}
 
 In particular, this applies to the situation where $X=Y$ is a $2$-sphere $S^2$ and $f$ is a branched covering map $f$ on  $S^2$. In this context,  a similar statement is also true for open connected sets: if $V\sub S^2$ is a region and $U\sub S^2$ a component of $f^{-1}(V)$, then $f(U)=V$ (see 
 Lemma~\ref{lem:proper}~\ref{item:proper2}).  
 
  Lemma~\ref{lem:pre_conn_comp}   was proved in 
  \cite[Theorem~7.5, p.~148]{Why}. Since it is a bit hard to see the main ideas  of the argument  in this reference, we decided to include a proof. 
    We need    the
following fact. 

\begin{theorem}[\v{S}ura-Bura]
  \label{thm:Sura-Bura}
  \index{Sura-Bura theorem@\v{S}ura-Bura theorem}
  Let  $X$  be a compact metric space. 
  Then every component $C$ of $X$ is the intersection of  all
  clopen subsets of $X$ that contain $C$.
\end{theorem}

Here a 
{\em clopen}\index{clopen} 
subset of a topological space  is a
set that is both open and closed.
Proofs of Theorem~\ref{thm:Sura-Bura} can be found in
\cite[Corollary~1.34]{Bur0}
and \cite[Appendix to Chapter~14]{Rem}. 
The  theorem is in fact still true for locally compact
Hausdorff spaces. 
We will use a slight variant of the \v{S}ura-Bura theorem.

\begin{cor}
  \label{cor:Sura-Bura}
  Let $X$ be a compact metric space and
  $C$ be a component of $X$. Then there is a nested sequence
  $A_1\supset A_2 \supset \dots $ of clopen subsets of $X$ such
  that $C= \bigcap_n A_n$.
\end{cor}

\begin{proof}  Let $C$ be a component of $X$.  We  define 
  \begin{align*}
    \mathcal{A} \coloneqq&\; \{A \subset X: 
    C\subset A,   
    \text{ $A$  clopen in } X\}, \text{ and}
    \\ 
    \mathcal{B} \coloneqq&\; \{B\subset X : 
    B \cap C = \emptyset, 
    \text{ $B$ clopen in } X\}
    \\
    =&\; \{X\setminus A : A\in \mathcal{A}\}.
  \end{align*}
Then   $\bigcap_{A\in  \mathcal{A}}A =
  C$  by the \v{S}ura-Bura theorem, which is equivalent to   $\bigcup_{B\in \mathcal{B}}B = X \setminus C$.

  Now for each $n\in \N$ we  consider the set $$K_n\coloneqq  \{ x\in X : \dist(x, C)\geq
  1/n\}.$$  This is a compact subset of $X$
  that is disjoint from $C$.  Thus it is covered by finitely many
  sets in $\mathcal{B}$. Since  $\mathcal{B}$ is stable under taking finite unions of sets in $\mathcal{B}$, there exists {\em one} set $B_n \in \mathcal{B}$ with $K_n\sub B_n$. 
 
 Now define $A_n= X\setminus (B_1\cup \dots \cup B_n)$ for $n\in \N$. Then $A_n$ is clopen, $C\sub A_n$,  and $A_n \supset A_{n+1} $ for $n\in \N$. 
 Moreover, we have 
     \begin{equation*}
    C\subset \bigcap_n A_n 
    \subset 
    \bigcap_n (X\setminus B_n)   \subset 
    \bigcap_n (X\setminus K_n) =C.  
  \end{equation*}
  Here the last equality follows from the fact that $C$ is closed. We conclude that 
   $\bigcap_n A_n=C$ as desired.   
\end{proof}

\begin{proof}[Proof of Lemma~\ref{lem:pre_conn_comp}]
  Under the given assumptions,  consider the set 
  $Z\coloneqq
  f^{-1}(K)$, and  let $C$ be a component of $Z$. Then $Z$ equipped with the
 restriction of the  metric on $X$  is a compact metric
 space itself, and we can  apply
  Corollary~\ref{cor:Sura-Bura} to $Z$. 
  
  Let $A\subset Z$ be  clopen  in $Z$. Then $f(A) \subset K$ is
  clopen in $K$. Indeed,  since $A$ is open in $Z$,   there exists 
   an open set
  $U\subset X$ such that $A= U\cap Z$.  Then  $f(A) = f(U)\cap K$, 
  because $y\in f(U)\cap K$ if and only if there exists $x\in U\cap f^{-1}(K)=A$ with $f(x)=y$. 
 Since $f$ is an open map, $f(U)$ is open. This means
  that $f(A)=f(U)\cap K$ is open in $K$.

  Similarly, $A$ is closed in $Z$, and hence a compact subset of $X$. 
 Therefore, $f(A)\sub K$ is compact, and hence closed in $K$.

 In particular,  if $A\subset Z$ is non-empty and clopen in $Z$, then $f(A)$ is non-empty and clopen in $K$. This implies that 
  $f(A) = K$, since $K$ is connected. 

  Now let $A_1\supset A_2 \dots $ be a decreasing sequence of
  clopen sets in $Z$ with $\bigcap_n A_n = C$, as in
  Corollary~\ref{cor:Sura-Bura}. Then  $f(A_n) = K$ for each 
  $n\in \N$ by what we have just seen. 
  
   Now let $p\in K$ be arbitrary. Then for each $n\in \N$ there 
   exists $q_n\in A_n$ such that $f(q_n)=p$. Since $X$ is compact, by passing to a subsequence if necessary, we may assume that the sequence $\{q_n\}$ converges, say $q_n \to q\in X$ as $n\to \infty$.
   Then $f(q)=\lim_{n\to \infty} f(q_n)=p$ by continuity of $f$. 
   
 Since the sets $A_n$ are decreasing, we have $q_k\in A_n$, whenever $k\ge n$. Since $A_n$ is closed in $Z$ and hence also in $X$, this implies that $q\in A_n$ for each $n\in \N$.  Then
  $q\in \bigcap_n A_n= C$, and so   $p=f(q)\in f(C)$. Since $p\in
  K$ was arbitrary, 
  we conclude that
  $f(C) =K$ as desired. 
\end{proof}

Before we formulate the next lemma, we discuss some simple facts about branched covering maps between regions in $S^2$. 
Suppose $V\sub S^2$ is a {\em finitely connected region}, i.e., a region with finitely many 
complementary components.  Then the 
{\em Euler characteristic}\index{Euler characteristic} 
$\chi(V)$ of $V$ is given by  
$\chi(V)=2-k_V$, where $k_V\in \N_0$ is the number of complementary components of $V$. Note that $k_V$ is also equal to the number of components of $\partial V$. 

Let $f\: S^2\ra S^2$ be a branched covering map and $U$ be a connected component of $f^{-1}(V)$. Then $f(U)=V$ and the map $f|U\: U \ra V$ is proper 
(see Lemma~\ref{lem:proper}~\ref{item:proper2}). 
  Moreover, 
  each point $q\in V$ has the same number of preimages under $f$ in $U$ if we count multiplicity given by the local degree of $f$ at a preimage point.
This number is called the {\em degree of $f$ on $U$} and denoted by 
$\deg(f|U)$. So
$$ \deg(f|U)=\sum_{p\in U\cap f^{-1}(q)}\deg_f(p)$$
for each $q\in V$.
By a  variant of the 
Riemann-Hurwitz formula\index{Riemann-Hurwitz formula}
we have
\begin{equation}\label{eq:RiemHurVar} 
\deg(f|U)\cdot \chi(V)= \chi(U)+\sum_{c\in U\cap \crit(f)}(\deg_f(c)-1). 
\end{equation}
Implicitly  this includes the statement that  $U$ is also a finitely-connected region.

All  of this is  well known if $f$ is a rational map on the
Riemann sphere (see, for example, \cite[Section 5.4]{Be}); these
facts are also true for a  general branched covering map $f$,
since we can reduce to   rational maps by Corollary \ref{cor:brcovratup}. 

In particular, if $f|U\: U \ra V$ is a covering map,  then 
$$\deg(f|U)\cdot \chi(V)=\chi(U). $$ 
If here $k_V=2$, then $V$ is called a  {\em ring domain}. In this case, $\chi(V)=0$, which implies that $\chi(U)=0$. In other words, a  finite cover of a ring domain under $f$   is also a ring domain.

\begin{lemma}
  \label{lem:Jordan_comp}
  Let $f\colon S^2\to S^2$ be a branched covering map, 
  $K\subset S^2$ 
  be a compact connected set, $C\sub S^2$ be a component of 
  $f^{-1}(K)$, and $V\sub S^2$ be a  Jordan region with 
  $K\sub V$ such that $\overline V\setminus K$ contains no critical value of $f$. 
    Suppose that neither $K$ nor  $C$ separates $S^2$. 
    
    Then the unique 
    component $U$ of $f^{-1}(V)$ that contains  $C$ is a Jordan region that contains no other component of  $f^{-1}(K)$. 
    Moreover, for the degree of the proper map $f\: U\ra V$ we have  
    \begin{equation}\label{eq:degfUCcr} 
    \deg(f|U) =1+ \sum_{c\in C\cap \crit(f)} (\deg_f(c)-1). 
    \end{equation}  
 \end{lemma}
 
\begin{proof} Note that $C$ is a connected set in $f^{-1}(K)\sub
  f^{-1}(V)$, and so is contained in a unique component $U$ of  $f^{-1}(V)$. In the ensuing argument, 
we will actually define $U$ in a different, less direct way.
 This  will make it easier to establish our claims.

We start by considering the open set  $R\coloneqq V\setminus K$. It  has two complementary components, namely $K$ and the closed Jordan region $S^2\setminus V$. Since neither $K$ nor $S^2\setminus V$ separate $S^2$, their union $S^2\setminus R$ does not separate $S^2$ either 
(this follows from Janiszewski's lemma; see Lemma~\ref{lem:Janiszewski}). So $R$ is connected and hence a ring domain.

We can find a connected component  $R'$ of $f^{-1}(R)$ such that $C\cap \partial R' \ne \emptyset$. To see this, we run along some path in $V$ from a point in 
$ R\sub V$ towards $K$ until we first hit $K$. In this way, we can find a path $\ga$ in $V$ 
whose endpoint $y$ lies in $K$, but that has no other points with $K$ in common. 
By Lemma~\ref{lem:pre_conn_comp} we have $f(C)=K$ and so we can find 
a point $x\in C$ with $f(x)=y$. We can lift the path $\ga$ by   $f$ to a path $\alpha$ 
that ends in $x$ (see Lemma~\ref{lem:liftsofpathsbranched}).  Then 
$$f(\alpha\setminus\{x\})= \ga\setminus\{y\}\sub V\setminus K=R, $$
and so the connected set $\alpha\setminus\{x\}$ must lie in a component $R'$ of $f^{-1}(R)$. Then $x\in C\cap \partial R'$. 

By a  similar path lifting argument one can also see that there exists a component $J$ of $f^{-1}(\partial V)$ such that $J\cap \partial R'\ne \emptyset$.

Since $\partial V$ is a Jordan curve that does not contain any critical values of $f$,
all components of $f^{-1}(\partial V)$, and in particular $J$,  are Jordan curves. The sets $J$, 
$R'$, and $C$, are all disjoint, because $f$ maps them to the disjoint sets 
$\partial V$, $R$, and $K$, respectively. In particular, the connected sets $R'$ and 
$C$ must each  be contained in one of the two complementary components  of $J$. 
Let $U$ be the complementary component of $J$ that contains $R'$. Then 
$U$ is a Jordan region and we also have  $C\sub U$, as follows from $ C \cap\partial R'\ne \emptyset$.

We know that    $f(R')=R$, 
since the set
$R'$ is a component of 
$f^{-1}(R)$ (see Lemma~\ref{lem:proper}~\ref{item:proper2}).  Moreover,  $R'$ contains no critical points of $f$, since $R=f(R')=V\setminus K$ contains no critical value of $f$. This implies that  $f|R'$ is a covering map of $R'$ onto $R$. 

To see this, let $q\in R$ be arbitrary. We have to find a neighborhood of $q$ that is evenly covered by the map $f|R'$. For this  
we choose a small topological disk $D\sub S^2$ with $q\in D\sub R$ that is evenly covered by the branched covering map $f$ as in  Definition~\ref{def:brcovmap} (see 
Lemma~\ref{lem:evennei}). If $D'$ is a component of $f^{-1}(D)$ and $p'$ the unique point in $D'$ with $f(p')=q$, then either $D'\cap R'=\emptyset$ or $D'\sub R'$ and $\deg_f(p')=1$.
So in the latter case $f$ is a homeomorphism of $D'$ onto $D$. It easily follows that $D$ is 
 evenly 
covered by the map $f|R'\: R'\ra R$ 
in the sense of (unbranched) covering maps.

Since $f|R'\: R'\ra R$ is a covering map and $R$ is a ring domain, we conclude that $R'$ is a ring domain as well (see the discussion before the statement of the lemma). This implies that the boundary $\partial R'$ of $R'$ has precisely two connected components 
$B_1$ and $B_2$.  It follows from 
Lemma~\ref{lem:proper}~\ref{item:proper2}   that 
\begin{equation*}
  f(B_1)\cup f(B_2) 
  =
  f(B_1\cup B_2)
  =
  f(\partial R')
  \sub 
  \partial R
  \sub 
  K \cup \partial V.  
\end{equation*}
The sets $K$ and $\partial V$ are compact and disjoint, and the
sets $f(B_1)$ and $f(B_2)$ are connected. Hence each of the sets  $f(B_1)$ and $f(B_2)$ is completely contained in one of the sets $K$ or  $\partial V$. 

Since $C\cap \partial R'\ne \emptyset$, one of the sets $B_1$ or $B_2$ must meet $C$, say $C\cap B_1\ne \emptyset $. Since $f(C)=K$, this forces $f(B_1)\sub K$ by 
what we have just seen. Then $C\cup B_1$ is a connected subset of $f^{-1}(K)$. Since $C$ is a component of  $f^{-1}(K)$, it follows that $B_1\sub C$.

We also know that $J\cap \partial R'\ne \emptyset$, and so one of the sets $B_1$ or $B_2$ must meet $J$. Since $B_1\sub C$, we necessarily have $B_2\cap J\ne 
\emptyset$. This forces $f(B_2)\sub \partial V$, and by a similar argument as for 
$B_1$ we see that $B_2\sub J$. 

We now consider the set $U\setminus C$. This is a ring domain, because its complement has the two connected components $C$ and $S^2\setminus U$, whose union does not separate $S^2$ by Janiszewski's lemma. 
We know that $R'$ is open and that $R'\sub U\setminus C$. Moreover, $R'$ is  relatively closed in 
$U\setminus C$, because 
$$ \partial R'=B_1\cup B_2\sub C\cup J= C \cup \partial U, $$
and so $R'$ has no boundary points in $U\setminus C$. Since $U\setminus C$ is connected, we conclude that $R'=U\setminus C$. 

In particular, $U=R'\cup C$ and so 
$$ f(U) \sub f(R')\cup f(C)\sub V. $$ 
This implies $U\sub f^{-1}(V)$. Now  $U$ is a Jordan region and hence connected.
Moreover, it is a maximal connected set in  $f^{-1}(V)$, because any point 
in  $f^{-1}(V)$ not in $U$ is separated from $U$ by the Jordan curve $J\sub f^{-1}(\partial V)$   
that lies in the complement of $f^{-1}(V)$. Hence $U$ is a component 
of $f^{-1}(V)$. Moreover, $C$ is the only component of $f^{-1}(K)$ contained in $U$, because $U=R'\cup C$ and $f(R')=R=V\setminus K$, which implies that 
$R'$ is  disjoint from 
$f^{-1}(K)$. 

Finally, \eqref{eq:degfUCcr} follows from the Riemann-Hurwitz formula \eqref{eq:RiemHurVar}. Indeed, 
$U$ and $V$ are Jordan regions and so $\chi(U)=\chi(V)=1$; moreover,  the only 
critical points of $f$ in $U=R'\cup C$ are those contained in $C$, because $R'$ does not contain any. 
\end{proof}

\section{Strongly invariant equivalence relations}
\label{sec:quot-thurst-maps}

We now consider the question when a given Thurston map
 descends to a quotient map
 that
is itself a Thurston map. More precisely,  the setting is as
follows. Let $\sim$ be an equivalence relation on a $2$-sphere $S^2$. As before, 
we denote by $[x]$ the equivalence class of a point $x\in S^2$. 
 Let $\widetilde{S}^2\coloneqq S^2/\Sim$ be the quotient
space equipped with the quotient topology and  
$\pi \colon S^2\to \widetilde{S}^2$ be the  quotient map 
given by $ \pi(x)= [x]\in \widetilde{S}^2$ for 
$x\in S^2$. In this section, $\sim$ will often be of Moore-type, in which case the quotient space 
$\widetilde{S}^2$ is also a $2$-sphere.

Condition \eqref{eq:eq_f_inv} is necessary for $f$ to descend
to a Thurston map $\widetilde{f}$. However, this condition is
not sufficient even when $\sim$ is of Moore-type, as the following example shows. 

\begin{ex}
  \label{ex:ftilde_not_Thurston}
  Let $f\colon \CDach \to \CDach$ be the Thurston map given by
  $f(z)=z^2$. Let $\sim$ be the equivalence relation on $\CDach$
  that is obtained by collapsing the positive real line
  $[0,\infty] \subset \CDach$ to a point, meaning that
 $$ 
    x\sim y  \ :\Leftrightarrow \    x,y\in [0,\infty]
   \text{ or } x=y
  $$ 
  for  $x,y\in \CDach$. Clearly, this is an equivalence
  relation that is $f$-invariant and of Moore-type. Thus by
  Theorem~\ref{thm:moore} (Moore's theorem) the
  quotient $\widetilde{S}^2\coloneqq\CDach/\Sim$ is a $2$-sphere,
  and by Lemma~\ref{lem:f_descends} there 
  is a continuous map $\widetilde{f}\colon \widetilde{S}^2\to
  \widetilde{S}^2$ as in  \eqref{eq:sim_descends}.   However, the  map $\widetilde{f}$
   is not a branched covering map, and hence not a Thurston 
  map. Indeed,  note that in $\widetilde{S}^2$ all points $[x]$
  with $x\in 
  (-\infty,0]$ are distinct, but $\widetilde{f}$ maps each such
  point to $[x^2]=[0]\in \widetilde{S}^2$. Thus the point $[0]\in
  \widetilde{S}^2$ has infinitely many preimages under
  $\widetilde{f}$, which is
  impossible for a branched covering map on  the $2$-sphere $\widetilde{S}^2$.

  The map $\widetilde{f}$ can be described as follows. First
  $\widetilde{f}$ collapses
  an equator of the sphere $\widetilde{S}^2$ to a point. This
  results in two 
  topological $2$-spheres that are connected at one point. Then
  $\widetilde{f}$ maps each of these two spheres to the sphere
  $\widetilde{S}^2$ by
  orientation-preserving homeomorphisms.  
\end{ex}

In contrast, it can happen that a Thurston map descends to a Thurston map on a $2$-sphere quotient $S^2/{\sim}$ even though  $\sim$ is not of 
Moore-type.

\begin{ex}
  \label{ex:quotient_not_Moore}
  We again consider the rational Thurston
  map  $f\colon \CDach \to \CDach$   given by $f(z)=z^2$. Let $\sim$ be the equivalence relation
  on $\CDach$ defined as  $z\sim w$ if and only if $w=\pm
  z$ for $z,w\in \CDach$. Clearly,
   all equivalence classes except $[0]$ and $[\infty]$
  are disconnected and  so $\sim$ is not of Moore-type.
  

Since the equivalence relation $\sim$ is $f$-invariant,  
we know that
 $f$ descends to a map $\widetilde{f}\colon \CDach/{\sim} \to
  \CDach/{\sim}$  as in \eqref{eq:sim_descends}.
  We claim that $\CDach/{\sim}$ is a $2$-sphere and 
  $\widetilde{f}$ is topologically conjugate to $f$. 
  In particular, $\widetilde{f}$ is also 
  a Thurston map. 

Indeed,  if   $ \overline{\Halb} = \{z\in \CDach : \imag(z) \geq
  0\}$ denotes  the closed upper half-plane,  then $\CDach/{\sim} =
 \overline{\Halb}   /{\sim}$ 
 and  $\sim$ identifies the points $x$ and $-x$ for $x\in (0,\infty)\sub \partial  \Halb $ and no other points in $ \overline{\Halb} $. This implies  that  
   $\CDach/{\sim}$ is  a  $2$-sphere. An
  explicit homeomorphism $h\colon\CDach \to \CDach/{\sim}$ is
  given by $h(z) \coloneqq [\sqrt{z}\,]\in \CDach/{\sim}$ for 
  $z\in \CDach$, as can easily be verified.

  Let $\pi\colon \CDach \to \CDach/{\sim}$ be the quotient
  map. Then 
   $\widetilde{f}\colon \CDach/{\sim} \to \CDach/{\sim}$
  is given by
  $\widetilde{f}([z])= \widetilde{f}(\pi(z)) = \pi(f(z))= [z^2]$
for $z\in \CDach$.  
  Thus 
  $(h\circ f)(z) = [\sqrt{z^2}\,] = [(\sqrt{z})^2] = (\widetilde{f}
  \circ h)(z)$ for $z\in \CDach$.
  This means that $f$ and $\widetilde{f}$ are topologically conjugate 
  by the homeomorphism $h$.
\end{ex}

Our general criterion for obtaining branched covering  maps $\widetilde f$ on quotients $S^2/{\sim}$ as formulated  in 
Theorem~\ref{thm:f_descends_branched_cover}
uses the notion of  a strongly $f$-invariant equivalence relation 
 as in Definition~\ref{def:sim_strongly-inv}. The following
 statement puts this
 condition into perspective. 
\begin{lemma}[Strongly invariant equivalence relations]
  \label{lem:eq_strong_inv}
  \index{equivalence relation!f-invariant@$f$-invariant!strongly}
  \index{strongly f-invariant@strongly $f$-invariant equivalence relation}
\index{f-invariant@$f$-invariant!equivalence relation!strongly}\index{invariant!equivalence relation!strongly}
  Suppose  $f\colon S^2\to S^2$ is a branched covering map, and
  $\sim$ is an 
  $f$-invariant equivalence relation of Moore-type on
  $S^2$. Then the following conditions are equivalent: 
  \begin{enumerate}
  \item 
    \label{item:eq_strong_inv1}
    The equivalence relation $\sim$ is strongly $f$-invariant. 
   \item 
    \label{item:eq_strong_inv2}
   If $y\in S^2$, then  each  component
    of $f^{-1}([y])$ is a single equivalence class.
  \item 
    \label{item:eq_strong_inv3}
    If   $y\in S^2$, then   $f^{-1}([y])$ is a union of finitely many equivalence
    classes. 
  \item 
    \label{item:eq_strong_inv4}
    The induced map $\widetilde{f}$  in
    \eqref{eq:sim_descends} is discrete. 
   \end{enumerate}
\end{lemma}

Recall that $\widetilde{f}$ is \emph{discrete}\index{discrete}\index{map!discrete} means that
$\widetilde{f}^{-1}([y])$ is a discrete set in $S^2/{\sim}$ for
all $[y]\in S^2/{\sim}$. 
\begin{proof}
  \ref{item:eq_strong_inv1} $\Rightarrow$
  \ref{item:eq_strong_inv2} 
  Let $y\in S^2$ be arbitrary and $C$ be a component of
  $f^{-1}([y])$. If $x\in C$, then $[x]$ is connected, since $\sim$ is monotone,
   and $f([x])\sub [y]$, since $\sim$ is $f$-invariant. So $[x] \sub C$, showing 
    that $C$ is a  union
  of equivalence classes. Each such equivalence class
  $[x]\subset C$ is mapped by assumption
  \ref{item:eq_strong_inv1} to $[y]$, i.e., $f([x]) = [y]$. Thus
  $[x]$ contains a point from the finite set $f^{-1}(y)$, and so  $C$ consists of finitely many equivalence classes. Since $\sim$ is  of
  Moore-type, each equivalence class is a compact connected 
  set. A finite union of 
two or more such sets 
is disconnected. This implies  that  $C$ consists of a single equivalence class as
  desired. 

  \smallskip
  \ref{item:eq_strong_inv2} $\Rightarrow$
  \ref{item:eq_strong_inv3}
  Let $y\in S^2$ be arbitrary. By our assumption
  \ref{item:eq_strong_inv2}, each component of $f^{-1}([y])$ is
  an equivalence class $[x]$. Then  $f([x]) = [y]$ by Lemma~\ref{lem:pre_conn_comp}. 
  Thus $[x]$ contains a point from the finite set
  $f^{-1}(y)$. Hence there are only finitely many such
  equivalence classes, or components of $f^{-1}([y])$. 

  \smallskip
  \ref{item:eq_strong_inv3} $\Rightarrow$
  \ref{item:eq_strong_inv4} 
 Since  $\sim$ is $f$-invariant, there exists a well-defined continuous map  
  $\widetilde{f}\colon \widetilde{S}^2\to \widetilde{S}^2$ on $\widetilde S^2 = S^2/\Sim$ as in 
  \eqref{eq:sim_descends} (see
  Lemma~\ref{lem:f_descends}). 

Now consider an arbitrary point in $\widetilde{S}^2$ as given by an equivalence class   $[y]\in \widetilde{S}^2$, $y\in S^2$. Then by \eqref{eq:sim_descends} we have 
  $[x]\in \widetilde{f}^{-1}([y])$ for $x\in S^2$ if and only if
  $f([x]) \subset [y]$, or equivalently  $[x]\sub f^{-1}([y])$.   By assumption
  \ref{item:eq_strong_inv3}, there are only finitely many such
  equivalence classes $[x]$. Thus, $[y]$ has only  finitely many
  preimages under $\widetilde{f}$, and so  $\widetilde{f}$ is
  discrete.

  \smallskip
  \ref{item:eq_strong_inv4} $\Rightarrow$
  \ref{item:eq_strong_inv1} Let   $x\in S^2$ be arbitrary,  $y=f(x)$, and $C$ be the component of $f^{-1}([y])$ that contains $x$. Then $f(C)=[y]$  by
  Lemma~\ref{lem:pre_conn_comp}. 
  
   If $x'\in C$ is arbitrary, then $f(x')\in f(C)=[y]$, and so $f(x')\sim y$. The  $f$-invariance of $\sim$ implies
  $f([x'])\sub [f(x')]=[y]$, or equivalently $[x']\sub f^{-1}([y])$. Now $\sim$ is monotone, and so 
  $[x']$ is connected. We conclude that $[x']\sub C$. So $C$ is saturated, i.e., a union of equivalence classes. Each of these equivalence classes is mapped into 
  $[y]$, and hence a preimage of $[y]\in \widetilde{S}^2=S^2/\Sim $ under 
  $\widetilde{f}$.

Since   $\sim$ is of Moore-type, 
  the quotient space $\widetilde{S}^2=S^2/\Sim $ is a $2$-sphere. Moreover, since 
  $\widetilde{f}$ is  discrete by our assumption \ref{item:eq_strong_inv4}, it must be finite-to-one. So 
  $[y]$ can have only finitely many preimages under
  $\widetilde{f}$. By what we have seen, this implies that $C$
  consists of finitely many equivalence classes. Since $\sim$ is closed and monotone, each of these equivalence classes is compact and connected. So 
  $C$ can only be connected if it consists of a single equivalence class, i.e., $C=[x]$. Hence $f([x])=f(C)=[y]=[f(x)]$, and  \ref{item:eq_strong_inv1}  follows. 
  \end{proof}

To prove Theorem~\ref{thm:f_descends_branched_cover}, we first
consider the mapping behavior of $f$ near an  individual equivalence class.

\begin{lemma}
  \label{lem:sim_str_inv_map_eqclass}
  Suppose  $f\colon S^2\to S^2$ is  a branched covering map, and
   $\sim$ is  an equivalence relation of Moore-type on $S^2$
  that is strongly $f$-invariant.   
   Let $x\in S^2$  
  and $U'\subset S^2$ be a neighborhood of 
  $[x]$. Then there exists a neighborhood $U\subset U'$ of
  $[x]$ with the following properties:  
  \begin{enumerate}
  \item 
    \label{item:sim_str_inv_map_eqclass1}
    $U$ is a Jordan region.
  \item 
    \label{item:sim_str_inv_map_eqclass2}
    $U\setminus [x]$ does not contain any critical point of
    $f$.   
  \item 
    \label{item:sim_str_inv_map_eqclass3}
    The restriction $f|U\colon U\to f(U)$ is a proper map.
  \item 
    \label{item:sim_str_inv_map_eqclass4}
    If  $c_1,\dots, c_n$ are  the (distinct) critical points of $f$ contained
    in $[x]$, then the degree of $f$on $U$  is given by
    \begin{equation*}
      \deg(f|U) = 1 + \sum_{i=1}^n (\deg(f,c_i) -1). 
    \end{equation*}
    In particular, if $[x]$ does not contain any  critical point of $f$, then  the
map $f|U\colon U\to f(U)$ is a homeomorphism.
  \end{enumerate}
\end{lemma}

Recall that $\deg(f|U)$ was defined as the (constant) number of preimages of a point $q\in f(U)$ counting multiplicities (see the discussion before Lemma~\ref{lem:Jordan_comp}).

Note that under our 
given assumptions on $\sim$ and $f$, we immediately obtain the following implication  from 
Lemma~\ref{lem:sim_str_inv_map_eqclass}~\ref{item:sim_str_inv_map_eqclass4}:    
\begin{align}
  \label{eq:x_y_homeo}
  [x] \text{ does not} &\text{ contain a critical point of $f$} \Rightarrow
  \\ \notag
                       &\text  {$f$ is a homeomorphism of $[x]$ onto $f([x])$.} 
\end{align}

\begin{proof}
  Let $\sim$, $f$, $x\in S^2$, and $U'\subset S^2$ be as in the
  statement of the lemma. Moreover, let $c_1,\dots, c_n$ be
  the critical points of $f$  contained in $[x]$, and $y\coloneqq f(x)$.

 Since $\sim$ is strongly $f$-invariant, 
  we know that $f([x]) = [y]$.   By statements~\ref{item:eq_strong_inv2} and 
  \ref{item:eq_strong_inv3} in Lemma~\ref{lem:eq_strong_inv} the components 
  of $f^{-1}([y])$ are given by finitely many distinct equivalences classes, say $[x_1],\dots, [x_k]$. Here $[x]$ is one of these classes, and so we may assume $x_1=x$. 
  
   If we equip $S^2$ with a base metric that
  induces the  given  topo\-logy, then we can 
 choose $\eps>0$ so small that the neighborhoods $$\mathcal{N}_\eps([x_1]), \dots,
 \mathcal{N}_\eps([x_k])$$ are all disjoint and $  \mathcal{N}_\eps([x_1])= \mathcal{N}_\eps([x])\sub U'$. We can find a corresponding $\delta>0$ such  
  that $f^{-1}(\mathcal{N}_\delta([y]))\sub \mathcal{N}_\eps(f^{-1}([y]))$
  (see Lemma~\ref{lem:continv}). 
    
  We now choose  a Jordan region  $V\sub S^2$ such that   
  $[y]\sub V\sub \overline V\sub \mathcal{N}_\delta([y])$    and $\overline V\setminus [y]$ does not contain any critical value of $f$.   It is clear that  such a region $V$ exists if  $[y]$ is the  singleton set $\{y\}$.  If $[y]$ contains at least two points, then
  $[y]$ is a non-degenerate  continuum. The existence of  $V$ is then most easily 
  established  by identifying $S^2$ with
  $\CDach$ under some homeomorphism. Then by the Riemann mapping theorem there exists a conformal 
  map  $\varphi \colon \D \to \CDach \setminus [y]$.
 If we now define    
  $V= \CDach \setminus \varphi( \overline 
  B_\C(0, r))$ with $r\in (0,1)$ sufficiently close to $1$, then $V$ has the desired properties. 
  
Since $\sim$ is of Moore-type, neither $[y]$ nor any of the components 
$[x_1],\dots, [x_n]$ of 
$f^{-1}([y])$ separates $S^2$. If $U$ is the unique component  of $f^{-1}(V)$  that contains the component $[x]=[x_1]$ of $f^{-1}([y])$, then $U$ is a Jordan region by 
  Lemma~\ref{lem:Jordan_comp}. 
Moreover, by choice of $V$, we have $U\sub \mathcal{N}_\eps(f^{-1}([y]))$. Since   the connected set $U$ can only meet one of   the disjoint open  sets 
$\mathcal{N}_\eps([x_1]), \dots,
 \mathcal{N}_\eps([x_k])$, whose union is equal to   $\mathcal{N}_\eps(f^{-1}([y]))$, it follows that $U\sub \mathcal{N}_\eps([x])\sub U'$.

Now $U$ has clearly  properties 
  \ref{item:sim_str_inv_map_eqclass1} and  
  \ref{item:sim_str_inv_map_eqclass2}  as in the statement. 
 By  Lemma~\ref{lem:proper}~\ref{item:proper2} the map $f|U\: U\ra V$ is proper and $f(U)=V$.  Property
  \ref{item:sim_str_inv_map_eqclass3} follows. 
  
The identity for $\deg(f|U)$ follows from 
\eqref{eq:degfUCcr} in Lemma~\ref{lem:Jordan_comp}. If $[x]$ does not contain critical points, then $\deg(f|U)=1$ and so 
$f$ is a homeomorphism of $U$ onto $V=f(U)$.  
 Statement  \ref{item:sim_str_inv_map_eqclass4} follows. 
\end{proof}

We can now prove the main result of this chapter. 
\begin{proof}
  [Proof of Theorem~\ref{thm:f_descends_branched_cover}]
  Let $f\colon S^2\to S^2$ be a branched covering map, and
  $\sim$ be an 
  $f$-invariant equivalence relation on $S^2$ of
  Moore-type. 

  Assume first that $\widetilde{f}$ is a branched covering
  map. Then $\widetilde{f}$ is discrete, and thus $\sim$ is
  strongly $f$-invariant by Lemma~\ref{lem:eq_strong_inv}.

Conversely, suppose  that $\sim$ is strongly $f$-invariant.  In order to see that $\widetilde f$ is a branched covering map, we want to apply the criterion provided by 
Corollary~\ref{cor:brcovcrit}. 

First, the map $\widetilde f$ is continuous, and discrete
by condition \ref{item:eq_strong_inv4} in 
Lemma~\ref{lem:eq_strong_inv}.  To
see that $\widetilde{f}$ is also an open map, consider an
arbitrary open set $\widetilde{U}\subset \widetilde{S}^2$. Then
$U\coloneqq \pi^{-1}(\widetilde{U})\subset S^2$ is open and
saturated. Since $\sim$ is strongly $f$-invariant, 
 $f$ maps any saturated set to a saturated set. Since $f$ is
open, it follows that $V\coloneqq f(U)\subset S^2$ is open and
saturated, and so $\pi(V)\sub \widetilde S^2$ is open.  Now by \eqref{eq:sim_descends} we have
that 
\begin{equation*} 
  \widetilde{f}(\widetilde{U}) =
  \widetilde{f}(\pi(U)) = \pi(f(U)) =\pi(V), 
\end{equation*}
which implies that $\widetilde{f}(\widetilde{U})$ is open. Thus
$\widetilde{f}$ is an open map.

 With a suitable choice of a fundamental class on the $2$-sphere $\widetilde S^2$ we have  $\deg(\pi)=1$ (see Corollary~\ref{cor:psisoconq} and the subsequent discussion). Then 
 \begin{align*}
  \deg( \widetilde{f})&=\deg( \widetilde{f})\cdot \deg(\pi)=\deg(\widetilde{f}\circ \pi)\\
  &=\deg(\pi\circ f)= \deg(\pi)\cdot \deg(f)=\deg(f)>0. 
  \end{align*} 
In particular,  statement \ref{item:f_descends_deg} is true and $\widetilde{f}$ has positive degree.

In order to apply Corollary~\ref{cor:brcovcrit} and to conclude
that $\widetilde f$ is indeed a branched covering map, it remains to
show that $\widetilde f$ is a local homeomorphism in the
complement of some finite subset of $\widetilde S^2$.  We will do
this by an argument that will also establish formula
\eqref{eq:loc_deg_ftilde}.
  
Let $x\in S^2$ be arbitrary and $[x]$ the corresponding
equivalence class.  Let $U\subset S^2$ be a neighborhood of $[x]$
as in Lemma~\ref{lem:sim_str_inv_map_eqclass}. We set
$y\coloneqq f(x)$.  Then $[y]= f([x])$, since $\sim$ is strongly
$f$-invariant. Let $ U_s$ be the saturated interior of $U$.  Then
$[x]\sub U_s$ and by condition \ref{item:def_closed_sat_int} in 
Lemma~\ref{lem:closed_eq} the set
$U_s$ is open. Moreover, since $f$ is strongly invariant, the set
$V'\coloneqq f(U_s)$ is open and saturated. Let
$\widetilde U\coloneqq \pi(U_s)$ and
$ \widetilde V\coloneqq \pi(V')$.  Then $[x]\in \widetilde U$, $[y]=f([x])\in \widetilde V$,
the sets $\widetilde U$ and $ \widetilde V$ are open, and
$\widetilde f|\widetilde U\: \widetilde U\ra \widetilde V$ is
a continuous, open, and surjective map.

Let us consider a point in $\widetilde V$ distinct from $[y]$. It
is represented by an equivalence class $[y']\ne [y]$, where
$y'\in V'\sub f(U)$. Suppose that $x_1, \dots, x_k\in U$ are the
distinct preimage points of $y'$ under $f$ that lie in $U$; so
$\{x_1, \dots, x_k\}= U\cap f^{-1}(y')$.  Then none of these
points is a critical point of $f$ by choice of $U$ and so it
follows from
Lemma~\ref{lem:sim_str_inv_map_eqclass}~\ref{item:sim_str_inv_map_eqclass4}
that $$k= d_x \coloneqq 1 + \sum_{i=1}^n (\deg(f,c_i)-1), $$
where $c_1,\dots, c_n$ are the critical points of $f$ contained
in $[x]$.
  
Consider $i\in \{1, \dots, k\}$.  Since $\sim$ is strongly
$f$-invariant, we have $f([x_i])=[y']$.  This implies that
$[x_i]$ is contained in $U$. Indeed, otherwise the connected set
$[x_i]$ must meet the boundary of $U$ and so $U\cap [x_i]$ is not
relatively compact in $U$. On the other hand, $U\cap [x_i]$ is
contained in the subset $U\cap f^{-1}([y'])=(f|U)^{-1}([y'])$ of
$U$ which is compact, because $f|U\: U\ra f(U)$ is a proper
map. This is a contradiction showing that indeed $[x_i]\sub
U$.
This implies that actually $[x_i]\sub U_s$ and so
$[x_i]\in \widetilde U$.  Note that
  $$ \widetilde f([x_i])=(\widetilde f\circ \pi)(x_i)=(\pi\circ f)(x_i)= [y']. $$
  So the points $[x_1], \dots,  [x_k]\in \widetilde U$ are preimages of $[y']$ under 
  $\widetilde f$. It is clear that  $[y']$ cannot have other preimages in $\widetilde U$; indeed,
   suppose such a preimage is  represented by an equivalence class $[x']\sub  U_s$ distinct from 
  $[x_1], \dots,  [x_k]$. Then $f([x'])=[y']$ by strong $f$-invariance of $\sim$ and so there would be another preimage $x''\in [x']\sub U_s\sub U$ of $y'$ in $U$ distinct from the points $x_1, \dots, x_k$. 
  
The points $[x_1], \dots,  [x_k]\in \widetilde U$ are distinct; indeed,  
these equivalence classes are distinct from $[x]$ and so do not contain any critical points 
of $f$. So by   \eqref{eq:x_y_homeo} the map  $f$ is a homeomorphism of each of the
  equivalence classes $[x_1],\dots, [x_k]$ onto $[y]$. In particular, in each of these equivalences classes the point $y'$ has exactly one preimage which implies that the 
  equivalence classes $[x_1],\dots, [x_k]$ are distinct, because the points $x_1, \dots, x_k$ are. We conclude that $[y']$ has precisely $k=d_x$ preimages under $\widetilde f$ that lie in $\widetilde U$. 
  
  Now suppose in addition that $[x]$ does not contain critical points of $f$. Then $d_x=1$.
  Actually, then $f$ is a homeomorphism of $U$ onto $f(U)$ and our argument shows that 
  each point in $\widetilde V$ (including $[y]$) has precisely one preimage under $\widetilde f$ in $\widetilde U$. In this case, the map $\widetilde f|\widetilde U\: \widetilde U\ra \widetilde V$ is a continuous open bijection and hence a homeomorphism. 
  In particular, $\widetilde f$ is a local homeomorphism  near each point in $\widetilde S^2\setminus C$, where $C\coloneqq \pi(\crit(f))$ is a finite set. Corollary~\ref{cor:brcovcrit} now implies that $\widetilde f$
  is indeed a branched covering map on the $2$-sphere $\widetilde S^2$.
  
  Now that we know that $\widetilde S^2$ is a branched covering map, we return to the general case where we allow critical points of $f$ in $[x]$. By what we have seen,
  each point $[y']\ne [y]=\widetilde f([x])$ in $\widetilde V$ has precisely $k=d_x$ preimages 
  under $\widetilde f$ in $\widetilde U$. Here $\widetilde U$ can be chosen to 
  be contained in any  given neighborhood of $[x]$, because the 
  Jordan region $U$ that led to the definition 
  of $\widetilde U$ can be chosen to lie in an arbitrary neighborhood of $[x]$. In other words, each point $[y']\ne [y]$ close to $[y]$ has precisely $k=d_x$ preimages 
  under $\widetilde f$ close to $[x]$. Formula \eqref{eq:loc_deg_ftilde} for the local degree of 
  $\widetilde f$ at  $[x]$ follows.

  Now  \eqref{eq:loc_deg_ftilde}  immediately  implies that 
  $ \crit(\widetilde{f})=\pi(\crit(f))$. This, in combination with the identity  $\pi \circ f^n =
  \widetilde{f}^n \circ \pi$ for all $n\in \N$, gives that $
  \post(\widetilde{f})=\pi(\post(f))$. 
\end{proof}

\ifthenelse{\boolean{singlechapter}}{ 

%


\chapter{Combinatorially expanding Thurston maps}
\label{cha:combexp}

In Chapter~\ref{cha:subdivisions} we have  constructed  Thurston maps in a
geometric way from  two-tile subdivision
rules.  We want to know when the Thurston map 
realizing a subdivision rule  can be chosen to be  expanding. The key concept for an answer is the
notion of  combinatorial expansion. 

\begin{theorem}[Subdivision rules and expansion]
  \label{thm:combexp2}
  \index{combinatorially expanding}
  \index{expanding!combinatorially}
  \index{expanding}
  \index{Thurston map!expanding}
  \index{two-tile subdivision rule}
  \index{two-tile subdivision rule!combinatorially expanding}
  \index{subdivision}
  \index{two-tile subdivision rule!realization of}
  \index{map!realizing!subdivision}
  \index{realizing!subdivision}
  Let $(\DD^1, \DD^0,L )$ be a two-tile subdivision rule on a
  $2$-sphere $S^2$ that can be realized by a Thurston map $f\: S^2\ra
  S^2$ with $\post(f)={\bf V}^0$, where ${\bf V}^0$ is the vertex set
  of $\DD^0$.  Then $(\DD^1, \DD^0, L)$ can be realized by an
  expanding Thurston map if and only if $(\DD^1, \DD^0,L)$ is
  combinatorially expanding.
\end{theorem}

Recall that combinatorial expansion (see
Definition~\ref{def:combexprule}) for a two-tile subdivision rule means that  every  Thurston
map $f\: S^2\ra S^2$ realizing the subdivision rule is combinatorially expanding for the Jordan curve $\CC$ of $\DD^0$.  In this case, $\CC$ is $f$-invariant, 
 $\#\post(f)\ge 3$, $\post(f)\sub \CC$, and   there exists $n_0\in \N$
such that no $n_0$-tile for $ (f,\CC)$ joins opposite sides of
$\CC$ (see Definition~\ref{def:combexp}).

In general, one only has $\post(f)\sub {\bf V}^0$ for a Thurston
map $f$ realizing a subdivision rule as in
Theorem~\ref{thm:combexp2}. The stronger condition
$\post(f)= {\bf V}^0$ prevents the existence of additional
vertices in ${\bf V}^0$ that have no dynamical relevance and
force an additional normalization on the Thurston map. Without
the condition $\post(f)= {\bf V}^0$, Theorem~\ref{thm:combexp2} is
not true in general, as we will see  in
Example~\ref{ex:extra_vertex_D0}.

If a Thurston map $f\colon S^2\to S^2$ has an $f$-invariant
Jordan curve $\CC\subset S^2$ with $\post(f)\subset \CC$, then
for $f$ to be expanding it is necessary that $f$ is
combinatorially expanding for $\CC$ (this follows from
Lemma~\ref{lem:Dtoinfty}; see also Lemma~\ref{lem:no<3}).
The converse is not true in general: a combinatorially expanding
Thurston map need not be expanding. One still obtains a converse
if one allows a change of the map by a suitable isotopy.

\begin{theorem}[Expansion and combinatorial expansion]
\label{thm:combexp1}
\index{Thurston map!combinatorially expanding}
\index{combinatorially expanding}
\index{expanding!combinatorially}      
Let $f\:S^2\ra S^2$ be a Thurston map
that has an invariant Jordan curve $\CC\sub S^2$ with $\post(f)\sub \CC$. If $f$ is combinatorially   expanding for $\CC$,  then there is an orientation-preserving 
 homeomorphism $\phi\:S^2\ra S^2$ with  $\phi(\CC)=\CC$ that is isotopic to the
 identity on $S^2$  rel.~$\post(f)$ such  that 
   $g=\phi\circ f$ and  $\widetilde{g} = f\circ \phi$
  are expanding Thurston maps.  
\end{theorem}

Clearly, $g$ and $\widetilde{g}$ are both Thurston
 equivalent to $f$.  Note that
$\widetilde{g}= \phi^{-1} \circ g \circ \phi$; so  $g$ and $\widetilde{g}$ are topologically
conjugate. 
Moreover, $\post(g)= \post(\widetilde{g})=\post(f)$ as follows from Lemma~\ref{lem:T-eq_crit_post}. We also   have 
$g(\CC)\sub \CC$ and $\widetilde{g}(\CC)\sub \CC$
 (actually, it is not hard to see that even $g(\CC)=\widetilde{g}(\CC)=f(\CC)\sub \CC$). So the theorem says that if a
Thurston map $f$ is combinatorially expanding for an invariant
Jordan curve $\CC \subset S^2$ with $\post(f) \subset \CC$,
 then by
 ``correcting'' the map by post- or precomposing with  a suitable
 homeomorphism,  we can obtain an expanding Thurston map 
 with the same invariant curve and the same set of postcritical
 points.

The previous two theorems are  easy consequences of the following slightly more
technical result.

\begin{prop}
  \label{prop:combexp} 
  \index{combinatorially expanding}
  \index{expanding!combinatorially}
  \index{Thurston map!combinatorially expanding}
  \index{expanding}
  \index{Thurston map!expanding}
  Let $f\:S^2\ra S^2$ be a Thurston map that has an invariant Jordan
  curve $\CC\sub S^2$ with $\post(f)\sub \CC$. If $f$ is
  combinatorially expanding  for $\CC$,  then there exists  an expanding Thurston map $\widetilde f\: \widetilde S^2\ra \widetilde S^2$ that is Thurston equivalent to $f$ and has an 
$\widetilde f$-invariant Jordan curve $\widetilde \CC\sub \widetilde S^2$ with $ \post(\widetilde f)\sub  \widetilde \CC$.
Moreover, there exist    homeo\-morphisms
$h_0,h_1\: S^2\ra \widetilde S^2$ that are isotopic
rel.~$\post(f)$ and satisfy $h_0\circ f=\widetilde f\circ h_1$ 
as well as  $h_0(\CC)=\widetilde \CC= h_1(\CC)$. 
\end{prop}
 
So up to Thurston equivalence  every  {\em combinatorially expanding} Thurs\-ton map with an invariant Jordan curve can be promoted  to an {\em expanding} Thurston map with an invariant curve.

Proposition \ref{prop:combexp} shows that for a Thurston map with
an invariant Jordan curve
combinatorial expansion is \emph{sufficient} for the existence of an
 equivalent map that is expanding. One may ask whether 
combinatorial expansion is \emph{necessary} for this as well. The answer
 is negative, as we will see in Example~\ref{ex:exp_notcexp}.  The Thurston map  $f$ in this example  has an  invariant Jordan curve $\CC$ containing all its postcritical
points. It  is not combinatorially expanding for $\CC$ (and hence not
expanding), yet is equivalent to an expanding Thurston map $g$.

 On an intuitive level the assertion of Proposition~\ref{prop:combexp}  seems quite plausible. Namely, based on Proposition~\ref{prop:thurstonex} a map 
 $\widetilde f$ as 
 in Proposition~\ref{prop:combexp} can easily be constructed 
 if one can  find cell decompositions  with the same combinatorics as $\DD^n(f,\CC)$ where the cells are small in dia\-meter (with respect to a given background metric) when $n$ is large. Since $f$ is combinatorially expanding, Lemma~\ref{lem:submult} implies  that $D_n(f,\CC)\to \infty$ as 
$n\to \infty$. So if the level $n$ increases, one needs more and more tiles to form a connected set joining opposite sides of $\CC$, or more generally, to join  any two disjoint $k$-cells. 
Therefore,  it seems   evident that one should be able  to make  the
cells small while keeping their combinatorics the same. If one wants to
implement this idea, then one faces serious difficulties that make it
hard to  convert this  into a valid proof (in
\cite[Theorem~2.3]{CFP01} the authors claim a more general statement with
an argument  along these lines). 
 
For this reason our approach to the proof of
Proposition~\ref{prop:combexp} is different. For the map $f\: S^2\ra
S^2 $ in this proposition to be expanding, the intersection $\bigcap_n
X^n$ of any nested sequence $\{X^n\}$ of $n$-tiles should consist of
only one point (see Lemma~\ref{lem:charexpint}). In order to enforce
this condition, we introduce a suitable equivalence relation $\sim$ on
the sphere $S^2$ that collapses these intersections $\bigcap_n X_n$ to
points. We then use Moore's theorem (Theorem~\ref{thm:moore}) to show
that the quotient space $S^2/\Sim$ is also a $2$-sphere.  The map
$\widetilde f$ will  be the induced map on $S^2/\Sim$. It encodes the same combinatorial information as the original map $f$, because $f$ and $\widetilde f$ realize isomorphic two-tile subdivision rules (see Corollary~\ref{cor:isosubdic}). In particular, these maps  are Thurston equivalent as can be deduced from  Lemma~\ref{lem:isotwotequiv} (we will actually give a different and more direct argument for this). While
our  approach is quite natural, it is somewhat lengthy to carry out
and will occupy the whole chapter.

%
%

In the following, $f\: S^2\ra S^2$ is  a Thurston map
and $\CC\sub S^2$  an $f$-invariant Jordan curve with $
\post(f)\sub \CC$ for which $f$ is combinatorially expanding.
We consider the cell decompositions $\DD^n=\DD^n(f,\CC)$
for $n\in \N_0$ as given by Definition~\ref{def:DDn}.  
 As before, we denote 
  by $\X^n$, $\E^n$ and ${\bf V}^n$ the set of $n$-tiles, $n$-edges, and $n$-vertices for $(f,\CC)$, respectively. 
  A subset $\tau \sub S^2$ is called a {\em tile} if it is an $n$-tile for some $n\in \N_0$. We use the terms {\em edge}, {\em vertex}, and {\em cell} in a similar way.  
  In particular, in this chapter  the term ``cell'' will always be used with  this specific meaning. We will use the term {\em topological cell}  to refer to the more general notion of cells as defined in
  Section~\ref{s:celldecomp}. 
    
  Since $\CC$ is $f$-invariant, $\DD^{n+k}$ is a refinement of $\DD^n$ for  $n,k\in \N_0$.  For each $X\in \X^{n+k}$ there exists a unique $Y\in \X^n$ with $X\sub Y$. Conversely, each $n$-tile $Y$ is equal to the union of all $(n+k)$-tiles  contained in $Y$, and similarly each $n$-edge $e$ is  equal to the union of all $(n+k)$-edges contained in $e$ (all this was proved in 
  Proposition~\ref{prop:invmarkov}). We will use this fact that cells are subdivided by cells of the same dimension and higher levels repeatedly in the following. 


\subsection*{The equivalence relation on $S^2$}
\label{sec:equiv-relat-s2}
 As before (see \eqref{eq:Xseq}), we denote by  $\mathcal{S}=\mathcal{S}(f,\CC)$  the set of all sequences $\{X^n\}$ with $X^n\in \X^n$  for  $n\in \N_0$ and 
\begin{equation*}
  X^0\supset X^1\supset X^2\supset \dots\,. 
\end{equation*}
We know (see Lemma~\ref{lem:charexpint})  that expansion of a Thurston
map with an invariant curve is characterized by the condition  that  
$\bigcap_n X^n$ is always a singleton set if $\{X^n\}\in
\mathcal{S}$. This may not be the case for our given map  $f$, and so
we want to identify all points in such an intersection $\bigcap_n
X^n$. This  will not lead to  an equivalence relation, since
transitivity may fail.  As we will see, this issue is resolved if we
define the relation as follows.

\begin{definition} \label{def:erel} Let $x,y\in S^2$ be arbitrary. We write 
   $ x\sim y$ if and only if for all  $\{X^n\}, \{Y^n\}\in \mathcal{S}$ 
      with $x\in \bigcap_n X^n$  and  $ y\in
      \bigcap_n Y^n$  we have   $X^n\cap Y^n\neq \emptyset$  for
        all  $n\in \N_0$.
     \end{definition}

Recall from \eqref{def:dk} that $D_n=D_n(f,\CC)$ denotes the minimal number 
of $n$-tiles required to form a connected  set $K^n$ joining opposite sides of $\CC$. 
 Since $f$ is combinatorially expanding for $\CC$ 
(see Definition~\ref{def:combexp}), we have $\#\post(f)\ge 3$ and so the term
``joining opposite sides''  is meaningful (see  Definition~\ref{def:connectop}).
Moreover, there exists $n_0\in \N$ such that $D_{n_0}(f,\CC)\ge 2$, and  so
by Lemma~\ref{lem:submultexp} we have $D_n=D_n(f,\CC)\to \infty$
as $n\to \infty$.  
In combination with Lemma~\ref{lem:flowerbds} this implies that
if $\tau, \sigma$ are disjoint $k$-cells, $K^n$ is a connected
set of $n$-tiles with $\sigma\cap K^n\ne \emptyset$, and
$\tau\cap K^n\ne \emptyset$, then the number of tiles in $K^n$
tends to infinity and thus cannot stay bounded as $n\to \infty$.
We will  use this fact  in the proof of the following lemma. 

\begin{lemma} \label{lem:aeq}
The relation $\sim$ is an equivalence relation on $S^2$.
\end{lemma} 

\begin{proof} Reflexivity and symmetry of the relation $\sim$ are clear.
To show transitivity, let $x,y,z\in S^2$ be arbitrary and assume that 
$x\sim y$ and $y\sim z$. Let $\{X^n\}, \{Z^n\}\in \mathcal{S}$ 
      with $x\in \bigcap_n X^n$  and  $ z\in
      \bigcap_n Z^n$  be arbitrary. We have to show that $X^n\cap Z^n\ne \emptyset$ for all $n\in \N_0$.

If this is not the case, then there exists $n_0\in \N_0$ such that $X^{n_0}\cap Z^{n_0}= \emptyset$. To reach a contradiction,  pick a sequence  $\{Y^n\}\in \mathcal{S}$
      with $y\in \bigcap_n Y^n$. Since $x\sim y$ and $y\sim z$, we have 
      $X^n\cap Y^n\ne \emptyset$ and $Y^n \cap Z^n\ne \emptyset$ for all $n\in \N_0$. Then $X^{n_0}\cap Y^n\supset X^{n}\cap Y^n\ne \emptyset$
      and $Z^{n_0}\cap Y^n\supset Z^{n}\cap Y^n\ne \emptyset$
      for all $n\ge n_0$. 
         So the $n$-tile $Y^n$ connects the disjoint $n_0$-tiles 
      $X^{n_0}$ and $Y^{n_0}$ for all $n\ge n_0$. As we discussed, this is impossible by Lemma~\ref{lem:flowerbds}. 
\end{proof}

The following lemma gives  convenient characterizations when two points are equivalent.

\begin{lemma} \label{lem:erel}
Let $x,y\in S^2$ be arbitrary. Then the following conditions 
are equivalent:
\begin{enumerate} 
\item
  \label{item:erel1}
  $x\sim y$.

\item
  \label{item:erel2}
  There exist 
  sequences
  $\{X^n\}, \{Y^n\}\in \mathcal{S}$ with $x\in \bigcap_n
  X^n$, $ y\in \bigcap_n Y^n$, and $X^n\cap Y^n\neq \emptyset$ for all
  $n\in \N_0$.
        
\item
  \label{item:erel3}
  For all cells $\sigma,\tau \sub S^2 $ with $x\in \sigma$, $y\in
  \tau$, we have $\sigma\cap \tau \ne \emptyset$.
\end{enumerate}
\end{lemma} 

\begin{proof} The implication \ref{item:erel1}
  $\Rightarrow$~\ref{item:erel2} is clear. 

  To show the reverse implication \ref{item:erel2} $\Rightarrow$
  \ref{item:erel1}, we assume that there exist 
  sequences 
  $\{X^n\}, \{Y^n\}\in
  \mathcal{S}$ with $x\in \bigcap_n X^n$, $ y\in \bigcap_n Y^n$, and
  $X^n\cap Y^n\neq \emptyset$ for all $n\in \N_0$. We claim that if
  $\{U^n\}, \{V^n\}\in \mathcal{S}$ are two other sequences with $x\in
  \bigcap_n U^n$ and $y\in \bigcap_n V^n$, then $U^n\cap V^n\ne
  \emptyset $ for all $n\in \N_0$.  To reach a contradiction, assume
  that $U^{n_0}\cap V^{n_0} = \emptyset$ for some $n_0\in \N_0$. We
  then have
        $$U^{n_0}\cap X^{n}\supset \{x\} \ne \emptyset \quad \text{ and }  \quad V^{n_0}\cap Y^{n}\supset \{y\} \ne \emptyset  $$
        for all $n\in \N$. Moreover, $X^n\cap Y^n\ne \emptyset$, and
        so for each $n\in \N_0$, the set $K^n\coloneqq X^n\cup Y^n$ is connected,
        consists of two $n$-tiles, and meets the disjoint $n_0$-tiles
        $U^{n_0}$ and $V^{n_0}$.  As before, this contradicts
        Lemma~\ref{lem:flowerbds}.  Hence $x\sim y$ as desired.
     
        The implication \ref{item:erel3} $\Rightarrow$
        \ref{item:erel1} is again clear. To prove \ref{item:erel1}
        $\Rightarrow$ \ref{item:erel3}, suppose that $x\sim y$. We
        argue by contradiction and assume that there exist cells
        $\sigma$ and $\tau$ with $x\in \sigma$, $y\in \tau$, and $\sigma\cap
        \tau=\emptyset$. By subdividing the cells if necessary,  we may assume that 
  $\sigma$  and $\tau$  are cells  of the same level $n_0$.  
  
  There are sequences  $\{X^n\}, \{Y^n\}\in \mathcal{S}$ 
      with $x\in \bigcap_n X^n$ and  $ y\in
      \bigcap_n Y^{n}$. 
      Since $x\sim y$, we have $X^n\cap Y^n\ne \emptyset$ for all $n$. 
      
    This implies that   for $n\in \N_0$ the set $K^n=X^n\cup Y^n$ is connected and  consists of at most two $n$-tiles. Moreover,
   $$K^n\cap \sigma \supset X^n\cap \sigma\supset\{x\}\ne \emptyset,$$
      and similarly, $K^n\cap \tau\ne \emptyset$. Hence $K^n$ connects the disjoint $n_0$-cells $\sigma$ and $\tau$. Since $f$ is combinatorially expanding, this is again impossible by Lemma~\ref{lem:flowerbds} for large $n$. This gives the desired contradiction.  
        \end{proof}

The previous lemma implies that all points in an intersection
$\bigcap_n X^n$ with $\{X^n\}\in \mathcal{S}$ are equivalent.  It  is
clear that $\sim$ is the ``smallest'' equivalence relation with this
property. 
  
If $x\in S^2$ we denote by $[x]\sub S^2$ the equivalence class of $x$ 
with respect to the equivalence relation $\sim$, and
by $$\widetilde S^2=S^2/\Sim =\{[x]:x\in S^2\}$$  the quotient space of $S^2$ under $\sim$.  So  $\widetilde S^2$ consists of all equivalence classes of $\sim$.  Such an equivalence class is both a point in
 $\widetilde S^2$ and a subset of $S^2$.  We equip $\widetilde S^2$ with the quotient topology. Then the quotient map $\pi\: S^2\ra \widetilde S^2$, $x\in S^2\mapsto\pi(x)\coloneqq  [x]$,  is continuous. 

In order to prove that $\widetilde{S}^2$ is in fact a topological
$2$-sphere,  we want to show that $\sim$ is of Moore-type (see 
Definition~\ref{def:eq_Moore_type}) and apply
Theorem~\ref{thm:moore} (Moore's theorem). To verify the relevant conditions for 
$\sim$, we need a good geometric description of the equivalence
classes. To set this up,  consider a point  $x\in S^2$
and let $n\in \N_0$ be arbitrary. We define
\begin{equation}
  \label{eq:def_Omn}
  \Omega^n =\Omega^n(x)= \bigcup_{x\in c^n} \inte (c^n),
\end{equation}
where the union is taken over all $n$-cells $c^n$  that contain
$x$. Recall that $\inte(c^n)=c^n$ if $c^n$ has  dimension $0$. 

Note that
\begin{equation}
  \label{eq:closure_Omn}
  \overline{\Omega}^n= \bigcup_{x\in X^n} X^n,
\end{equation}
where the union is taken over all $n$-tiles $X^n$ that contain $x$. 
Indeed, every cell $c^n$ in the union in  \eqref{eq:def_Omn} is contained
in an $n$-tile $X^n$ that contains $x$. Therefore, 
$\overline{\Omega}^n\subset \bigcup_{x\in X^n} X^n$, since the set on the right
hand side  is closed.  
On the other hand,  for each $n$-tile $X^n$ containing
$x$ we have $\inte(X^n)\subset \Omega^n$. Thus $\bigcup_{x\in X^n}  X^n\subset
\overline{\Omega}^n$, and  \eqref{eq:closure_Omn} follows. 

\begin{lemma}
  \label{lem:Omn_simply_conn}
  The set $\,\Omega^n\subset S^2$ is a simply connected region.
\end{lemma}

\begin{proof}
  We have to consider three cases.  When $x=v$ is an $n$-vertex then
  $\Omega^n(v)=W^n(v)$ is the $n$-flower of $v$ by
  Definition~\ref{def:flower}. Recall from
  Lemma~\ref{lem:flowerprop}~\ref{item:flower_prop1} that such a
  vertex flower is a simply connected region.

 Suppose that $x$ is not an $n$-vertex, but $x$ is contained in an 
  $n$-edge $e^n$. Then $x$ is necessarily contained in the interior $\inte(e^n)$  of
  $e^n$, and in no other $n$-edge.  There are precisely two distinct $n$-tiles $X^n$ and $Y^n$
  that contain $e^n$ in their boundaries. These are all  the
  $n$-tiles that contain $x$. Thus
  $$\Om^n=\inte(X^n)\cup \inte(e^n) \cup \inte  (Y^n). $$ 
  Then $\Om^n$ is  a simply
  connected region (see Lemma~\ref{lem:specprop}~\ref{item:prop_cell4}). 
  
  Finally,  suppose  that $x$ is not contained in any $n$-edge. Then there
  is a unique $n$-tile $X^n$ that contains $x$. Then $\Omega^n=\inte
  (X^n)$ is an open Jordan region, and so simply connected. 
\end{proof}

Let $M\subset S^2$ be an equivalence class with respect to
$\sim$. We select a point $x\in M$ as follows.
\begin{enumerate}
\item[\emph{Case 1:}]
  \label{item:Omn_v}
  If $M$ contains a vertex $v$,  then $x\coloneqq v$.
\item[\emph{Case 2:}] 
  \label{item:Omn_e}
  If $M$ contains no vertex, but intersects an edge $e$, then we choose  a point in $M\cap e$ for $x$.
\item[\emph{Case 3:}]  
  \label{item:Omn_X}
  If $M$ contains no vertex and does not intersect any edge,  then 
  we choose 
   an arbitrary point in $M$ for $x$. 
\end{enumerate}
We say that $M$ is of 
\emph{vertex-type}\index{vertex-type} in the
first, of 
\emph{edge-type}\index{edge-type} in the second, and of
\emph{tile-type}\index{tile-type}\index{type of equivalence classes}
in the last case. We call the point $x$ a \emph{center} 
\index{center of equivalence class} 
of the equivalence class $M$.  By
Lemma~\ref{lem:erel} an equivalence class cannot contain two distinct
vertices. So if $M$ is of vertex-type, then its center is unique, but
this may not be true in the other two cases.

With such a choice of a center $x$ for given $M$, we define
$\Omega^n=\Omega^n(x)$ as in \eqref{eq:def_Omn} for  $n\in
\N_0$. Note that every $(n+1)$-cell $c^{n+1}$  with $x\in c^{n+1}$ is
contained in an $n$-cell $c^n$ with $x\in c^n$ and 
 $\inte (c^{n+1})\subset \inte (c^n)$ (see
Lemma~\ref{lem:mincell}). Thus $\{\Omega^n\}$ is a decreasing sequence
of sets, i.e.,
\begin{equation}
  \label{eq:Om0Om1}
  \Omega^0 \supset \Omega^1 \supset \Omega^2 \supset \dots\,. 
\end{equation}
The different types of equivalence classes are illustrated in
Figure~\ref{fig:equiv_types}. 

\begin{lemma}
  \label{lem:eclass}
  Let $M\sub S^2$ be an arbitrary equivalence class with respect to
  $\sim$, and  $\Omega^n$ be defined as above for  $n\in \N_0$. Then 
  \begin{equation} 
    \label{eq:Mbigcap}
    M=\bigcap_{n}\Om^n =\bigcap_{n}\overline{\Omega}^n.
 \end{equation}
\end{lemma} 

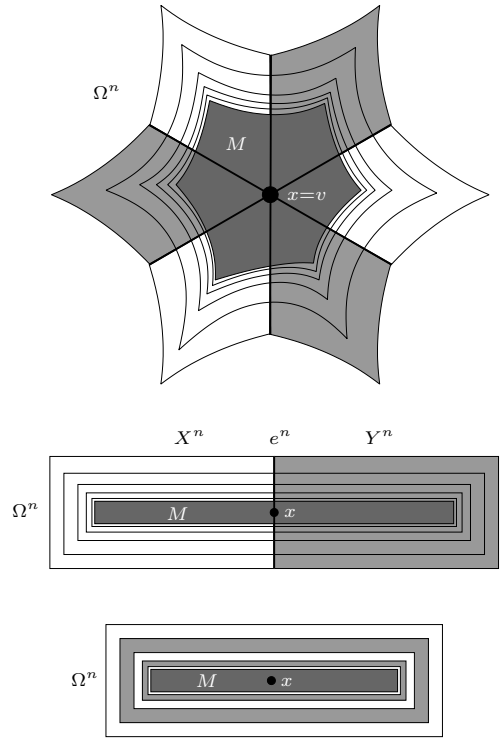
\begin{figure}
  \centering
  \begin{overpic}
    [width=6cm, tics=10,
    ]{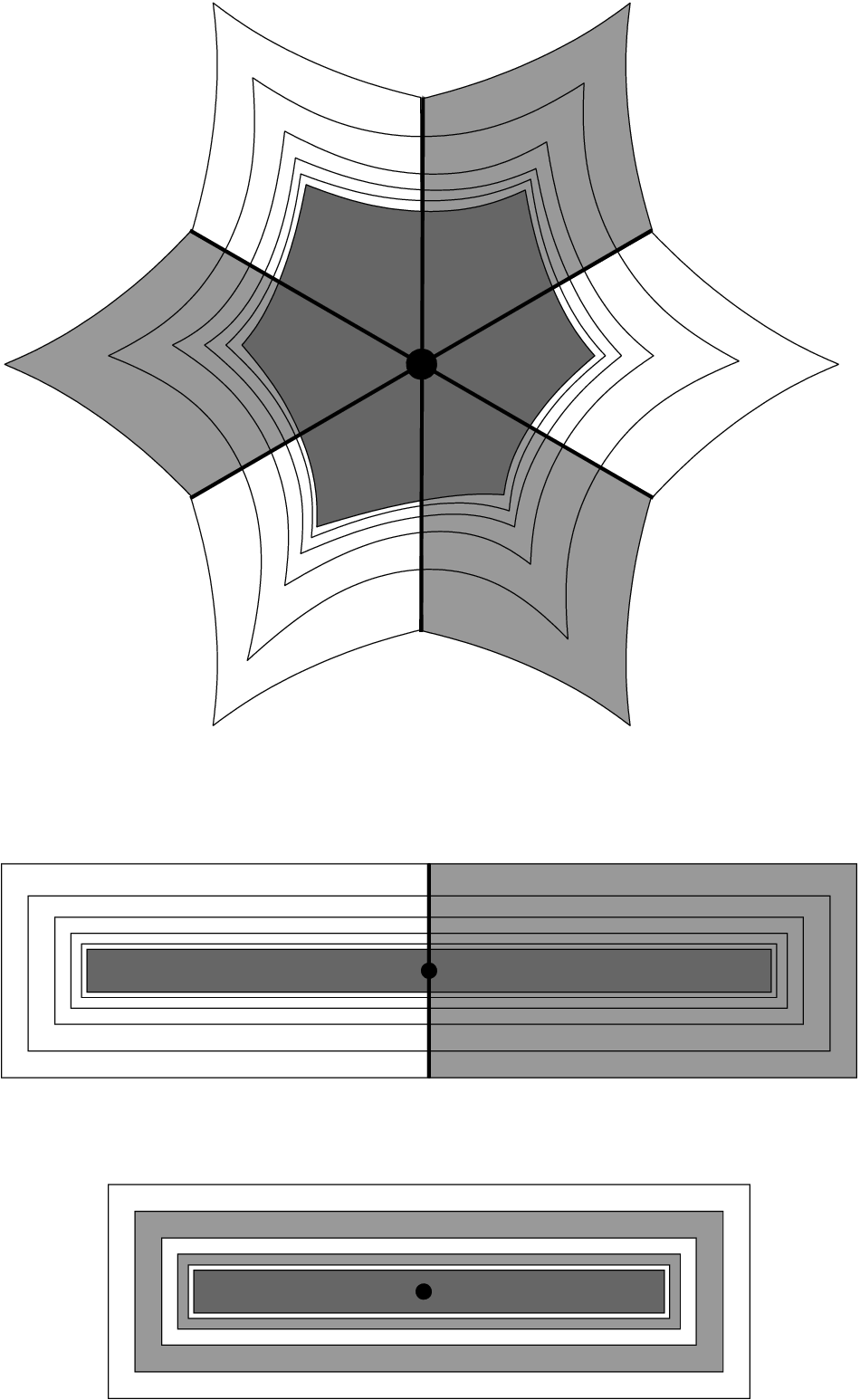}
    \put(6,87){$\scriptstyle \Omega^n$}
    {\color{white}
    \put(32.4,73.3){$\scriptstyle x=v$}
    \put(24,80){$\scriptstyle M$}
    }
    \put(-5,30){$\scriptstyle \Omega^n$}
    \put(17,40){$\scriptstyle X^n$}
    \put(43,40){$\scriptstyle Y^n$}
    \put(30,40){$\scriptstyle e^n$}
    {\color{white}
    \put(32,30){$\scriptstyle x$}
    \put(16,29.6){$\scriptstyle M$}
    }
    \put(3,7){$\scriptstyle \Omega^n$}
    {\color{white}
      \put(31.6,7.1){$\scriptstyle x$}
      \put(20, 6.8){$\scriptstyle M$}
    }
  \end{overpic}
  \caption{Equivalence classes of vertex-, edge-, and tile-type.}
  \label{fig:equiv_types}
\end{figure}

\begin{proof}
 Let $x\in M$ be a center of $M$. In order to establish the inclusion 
   \begin{equation}\label{eq:barOmsubM}
    \bigcap_n\overline{\Omega}^n\subset M, 
  \end{equation}
let $y\in \bigcap_n
  \overline{\Omega}^n$ be arbitrary. We have to show  that $x\sim y$ (which implies $y\in M$). If this is not the case, then $x\not\sim y$, and so there exist 
 sequences   $\{X^n\},\{Y^n\}\in \mathcal{S}$ with  
  $x\in \bigcap_{n}X^n$ and $y\in \bigcap_{n}Y^n$, and $n_0\in \N_0$ such that 
  $X^{n_0}\cap Y^{n_0}=\emptyset$. On the other hand, for each $n\in \N_0$ we have 
  $y\in \overline \Om^n$, and so  by 
  \eqref{eq:closure_Omn} we can find an $n$-tile $Z^n$ with $x,y\in Z^n$. Then 
  $X^{n_0}\cap Z^n\supset \{x\}\ne \emptyset$ and $Y^{n_0}\cap Z^n
  \supset \{y\}\ne \emptyset$, i.e., for all $n\in
  \N_0$ the tile $Z^n$
  intersects the disjoint tiles $X^{n_0}$ and $Y^{n_0}$.  This is impossible, since $f$ is combinatorially expanding
  for $\CC$ (see Lemma~\ref{lem:flowerbds}). We obtain a contradiction
  and \eqref{eq:barOmsubM} follows.

  To finish the proof, it is enough  to show that $M\subset \bigcap_n
\Omega^n$, or equivalently, that if 
$$y\in S^2\setminus \bigcap_n
\Omega^n=\bigcup_n(S^2\setminus \Om^n)$$ is arbitrary, then $y\not \sim x$ (and so $y\notin M$). Note that  the sets $S^2\setminus \Om^n$ for $n\in \N_0$  form  an increasing  sequence, and so 
 $y\notin \Om^n $ for all sufficiently large $n$.
In order to show  $y\not \sim x$,  
we  now consider 
three cases according to the type of $M$. 

\smallskip{}
\emph{Case 1:} $M$ is of vertex-type. 
In this case,  $x$ is a vertex, say an $n_0$-vertex, where $n_0\in \N_0$. 
Then $x$ is also an $n$-vertex for all
$n\geq n_0$. Fix $n\geq n_0$ such that $y\notin  \Om^n$. Then there
exists a unique $n$-cell  $\tau$ with  $y\in \inte(\tau)$. Since $y\notin  \Om^n$, we have $x\notin
\tau$. Now  $x$ is an $n$-vertex and so  $\{x\}$ is an $n$-cell, and the $n$-cells  $\tau\supset \{y\}$ and $\{x\}$
are disjoint. By Lemma~\ref{lem:erel} this implies $y\not\sim x$  as desired. 

 \smallskip{} 
 \emph{Case 2:} $M$ is of edge-type.   Then $x$ is contained in an edge, say 
 an $n_0$-edge $e^{n_0}$, where $n_0\in \N_0$.  By successive subdivisions 
 (see Proposition~\ref{prop:invmarkov}~\ref{item:invmarkov4}) we can find  $n$-edges $e^n$ for $n\geq n_0$
 that contain $x$  and that satisfy
 \begin{equation*}
   e^{n_0}\supset e^{n_0+1} \supset \dots\,.
 \end{equation*}

 Fix  $k\ge n_0$ such that $y\notin
 \Om^{k}$. Then there exists a unique $k$-cell  $\tau$ with $y\in \inte(\tau)$. Since 
 $y\notin \Omega^k$, we have $x\notin \tau$ and so $ e^k\not \sub \tau$.  
 Hence $ \tau \cap \inte(e^k)=\emptyset$ by
  Lemma~\ref{lem:celldecompint}~\ref{item:cell_decomp2}.  Let $u$ and $v$ be the endpoints of $e^{k}$.  Since
 these points are vertices, they do not belong to $M$ by
 assumption. So the set
  $$\bigcap_{n\geq n_0} e^n \sub \bigcap_n \overline {\Om}^{{\null}_{\scriptstyle n}}\sub M$$ 
  does not contain   $u$ or $v$ either. It follows that there exists $m\ge k$ such that $u,v\notin e^{m}$. Then $y\in \tau$,  $x\in e^m$, and  $\tau \cap e^m=\emptyset$, because
  $$\tau \cap e^{m}\sub \tau \cap (e^{k}\setminus \{u,v\})=\tau \cap \inte(e^k) =\emptyset. $$ By Lemma~\ref{lem:erel} this  implies  $y\not\sim x$  as desired. 

  \smallskip{} 
  \emph{Case 3:} $M$ is of tile-type. We pick  sequences $\{X^n\},\{Y^n\}\in \mathcal{S}$ with 
  $x\in \bigcap_nX^n$ and $y\in \bigcap_nY^n$. Since $x\in M$ and  $M$ does not meet any edges, for each $n\in \N_0$ the tile  $X^n$ is the unique $n$-tile with $x\in X^n$. Then  
  $x\in  \inte(X^n)$ and  $\Omega^n=\inte(X^n)$. In order to show that 
  $y\not \sim x$, we argue by contradiction and assume $y\sim x$. 
  Then  $K^n\coloneqq X^n\cap Y^n\ne \emptyset$ for each $n\in \N_0$ (see
  Definition \ref{def:erel}). The
  sets $K^n$, $n\in \N_0$, are non-empty nested compact sets.  Hence
  there exists a point $z\in \bigcap_n K^n$. Then $x\sim z$ and so
  $z\in M$.
  
  On the other hand, if $n_0\in \N_0$ is large enough, then $y\notin \Om^{n_0}=\inte(X^{n_0})$. Since $\partial X^{n_0}=X^{n_0}\setminus \inte(X^{n_0})$ consists of $n_0$-edges, $y\sim x$ and so $y\in M$, and $M$ does not meet edges, we then actually have $y\notin X^{n_0}$. Thus  $X^{n_0}\ne Y^{n_0}$. This means that  
 the intersection $K^{n_0}=X^{n_0}\cap Y^{n_0}$ consists of $n_0$-cells
  on the boundary of $X^{n_0}$ and of $Y^{n_0}$, and is hence
  contained in a union of $n_0$-edges. Since $z\in M\cap K_{n_0}$,  this
  implies  that $M$ meets an edge, contradicting our assumption in this case.
  So indeed $y\not \sim x$ as desired. 
  \end{proof}

The following consequence of the previous lemma will be one of the essential ingredients in the proof that
$\widetilde{S}^2$ is  a topological $2$-sphere. 

\begin{cor}
  \label{lem:erelcor}  
  Each equivalence class $M$ of $\sim$ is a compact connected set with
  connected complement  $S^2\setminus M$.  
\end{cor}
  
\begin{proof} 
  Let $M$ be an arbitrary equivalence class of $\sim$. Then 
  Lem\-ma~\ref{lem:eclass} and \eqref{eq:Om0Om1} imply that the
  set $M$ is the intersection of the nested 
  sequence of the  compact sets $\overline{\Omega}^n$, $n\in \N_0$.  It follows  from 
  \eqref{eq:closure_Omn}  that  each set $\overline{\Omega}^n$ is
  connected. Hence $M$ is also compact
  and connected.

  The complement $S^2\setminus \Om^n$ of the  open simply connected set $\Om^n$ (see Lemma~\ref{lem:Omn_simply_conn})  is
  connected.  So Lemma~\ref{lem:eclass} shows that the complement
  $S^2\setminus M$ of $M$ is the union of the  increasing sequence of
 the  connected sets $S^2\setminus \Om^n$. Hence $S^2\setminus M$ is connected.
\end{proof}

\subsection*{The quotient  space $\widetilde {S}^2$ is a topological $2$-sphere}
\label{sec:quot-space-widet}
After these preparations we are ready to show that $\widetilde {S}^2$
is a topological $2$-sphere.

\begin{lemma}
  \label{lem:simusc}   Let $\sim$ be the equivalence relation on $S^2$ as in Definition~\ref{def:erel}. 
  Then $\sim $ is of Moore-type and the quotient space $\widetilde S^2=S^2/\Sim$  is homeomorphic to $S^2$. 
  \end{lemma}

\begin{proof} 
  By Lemma~\ref{lem:aeq} our relation $\sim$ is indeed an equivalence
  relation. It remains to verify the conditions
  \ref{item:Moore1}--\ref{item:Moore4} in Definition~\ref{def:eq_Moore_type}. Then  
 $\widetilde S^2$ is a $2$-sphere by Theorem~\ref{thm:moore}
 (Moore's theorem).

  \smallskip
  {\em Conditions} \ref{item:Moore2} and \ref{item:Moore3} were already
  proved in Corollary~\ref{lem:erelcor}.   

  \smallskip 
  {\em Condition} \ref{item:Moore4}: There are at least two equivalence
  classes, because no two distinct vertices are equivalent by
  Lemma~\ref{lem:erel}, and each postcritical point of $f$  is a vertex
  (there are
  at least three such points).

  \smallskip 
  {\em Condition} \ref{item:Moore1}: Let $\{x_n\}$ and
  $\{y_n\}$ be convergent sequences in $S^2$ with $x_n\ra x$ and
  $y_n\ra y$ as $n\to \infty$, and suppose that $x_n\sim y_n$ for all
  $n\in \N$.  We have to show that $x\sim y$. Suppose this is not the
  case.  Then the equivalence classes $[x]$ and $[y]$ are disjoint.
  By Lemma~\ref{lem:Omn_simply_conn} and Lemma~\ref{lem:eclass} there
  exist simply connected nested 
  regions $\Omega^n_x$ and $\Omega^n_y$ for $n\in \N_0$ such that
$$[x]=\bigcap_n \Omega^n_x=\bigcap_n \overline{\Omega}^{n}_x \text{ and }
[y]=\bigcap_n \Omega^n_y=\bigcap_n \overline{\Omega}^{n}_y. $$ 
Since
$[x]$ and $[y]$ are disjoint, the sets $\overline{\Omega}^n_x$ and
$\overline{\Omega}^n_y$ will also be disjoint for sufficiently large
$n$, say $\overline{\Omega}^{n_0}_x\cap\overline{\Omega}^{n_0}_y
=\emptyset$. On the other hand, since $\Omega^{n_0}_x\supset[x]$ and
$\Omega^{n_0}_y\supset[y]$ are open, there exists $n_1\in \N$ such that
$x_{n_1}\in \Omega^{n_0}_x$ and $y_{n_1}\in \Omega^{n_0}_y$.  Since $\overline
\Omega^{{}_{\scriptstyle n_0}}_x$ and $\overline\Omega^{{}_{\scriptstyle n_0}}_y$
consist of $n_0$-tiles and are disjoint, this means that there exist
$n_0$-tiles $\sigma$ and $\tau$ with $x_{n_1}\in \sigma$, $y_{n_1}\in
\tau$, and $\sigma\cap \tau =\emptyset$. Hence $x_{n_1}\not \sim
y_{n_1}$ by Lemma~\ref{lem:erel}. This is a contradiction. It follows
that $\sim$ is closed. \end{proof}


\subsection*{Quotients of cells and the induced cell
  decompositions on $\widetilde S^2$}
\label{sec:quot-cells-induc}
We now study what happens
to our cells under the quotient map $\pi\: S^2 \ra \widetilde S^2$.
If $A\sub S^2$ is an arbitrary set, we denote by $\widetilde A$ its
image under  $\pi$.  So $\widetilde
A=\pi(A)=\{[x]:x\in A\}\subset \widetilde{S}^2$. We will see that if
$\sigma$ is an arbitrary cell (i.e., an element of $\DD^n$ for some
$n\in \N_0$), then $\widetilde \sigma$ is a topological cell of the
same dimension (Lemma~\ref{lem:qcells}).  Moreover, the images
$\widetilde \sigma$ of the $n$-cells $\sigma\in \DD^n$ form a cell
decomposition of $\widetilde S^2$ (Lemma~\ref{lem:cdcd}).

\begin{lemma}
  \label{lem:eq_intersections}
  Let $M$ be an arbitrary equivalence class with center $x\in M$. 
  If $\tau$ is an arbitrary cell, then
  \begin{equation*}
    \tau \cap M\ne \emptyset  
    \quad\text{if and only if}\quad 
    x\in \tau.
  \end{equation*}
\end{lemma}

\begin{proof}
  The ``if''-implication is obvious. 

  To show the other implication, assume that $\tau$ is a cell of level
  $n$ and $x\notin \tau$. Consider an $n$-cell $\sigma\subset
  \tau$. Then $x\notin \sigma$ and so $\inte (\sigma)$ is disjoint from
  $\Omega^n=\Omega^n(x)$ by \eqref{eq:def_Omn}, because  distinct
  $n$-cells have disjoint interiors. Recall from
  Lemma~\ref{lem:uniondisjint} that $\tau$ is the disjoint union of
  the interiors of all $n$-cells $\sigma\subset \tau$. Thus $\tau\cap
  \Omega^n=\emptyset$, and so  $\tau\cap M=\emptyset$ by
  Lemma~\ref{lem:eclass}. 
\end{proof}

The following lemma states that if we pass to the quotient space
$\widetilde S^2=S^2/\Sim$, then intersection and inclusion relations of cells 
are preserved. In particular, we do not create ``new''
intersections or inclusions between cells.

\begin{lemma}
  \label{lem:nonewint} 
  If $\sigma$ and $\tau$ are cells, then $\widetilde \sigma \cap
  \widetilde \tau =\widetilde {\sigma \cap \tau}$.  Moreover, we have
  $\widetilde \sigma\sub \widetilde \tau$ if and only if $\sigma\sub
  \tau$.
\end{lemma} 
 
\begin{proof} 
The inclusion $\widetilde {\sigma \cap \tau}\sub \widetilde \sigma
\cap \widetilde \tau$ is trivial. 
 
For the other inclusion consider an arbitrary point $[x]\in
\widetilde{\sigma}\cap \widetilde{\tau}\subset \widetilde{S}^2$. We
can assume that $x$ is a  center of $M=[x]\subset S^2$. 
Then $M$ meets
both cells $\sigma$ and $\tau$, and so $x\in \sigma \cap \tau$  by Lemma~\ref{lem:eq_intersections}. 
 Thus $[x]\in
\widetilde{\sigma\cap\tau}$. We have proved
$\widetilde{\sigma}\cap\widetilde{\tau} \subset \widetilde{\sigma \cap
  \tau}$ as desired.  

 In the second statement the implication $\sigma\sub \tau \Rightarrow 
 \widetilde \sigma\sub \widetilde \tau$ is trivial. For the other implication assume that  $\widetilde \sigma\sub \widetilde \tau$. Let  $k$ and $n$ be the levels  of $\sigma$ and $\tau$, respectively. For the moment, we make the additional assumption that $k\ge n$. 
 
  By Lemma~\ref{lem:vinint} there exists a vertex $v$  such that $v\in \inte(\sigma)$ (note that this is trivial if $\sigma$ is a $0$-dimensional cell).
 Then $[v]\in \widetilde \sigma \sub \widetilde \tau$. This   means that there exists a point $x\in \tau$ such that 
 $[v]=[x]$, or $v\sim x$. 
Condition~~\ref{item:erel3} in Lemma~\ref{lem:erel} implies that $\{v\}\cap
\tau\neq \emptyset$ or  $v\in \tau$. Thus
 \begin{equation}\label{eq:stint}
 \inte(\sigma)\cap \tau \ne \emptyset. 
 \end{equation}
 Since $k\ge n$, the cell decomposition $\DD^k$  containing  $\sigma$ is a refinement of the 
 cell decomposition $\DD^n$ containing  $\tau$. 
 Therefore, as we have seen in the first part of the proof of Lemma~\ref{lem:mincell}, the relation \eqref{eq:stint}  forces the inclusion
 $\sigma\sub \tau$. 
 
 If $k<n$, we subdivide $\sigma$ into cells of level $n$. By the previous argument, $\tau$ will contain each of these cells, and so we always have $\sigma\sub \tau$ as desired.  
 \end{proof}

 \begin{lemma}
   \label{lem:edgeinter} 
   Let $M\sub S^2$ be an equivalence class and $E\sub S^2$ be a finite
   union of edges. Then $E\cap M$ is connected. 
\end{lemma} 
   
\begin{proof} 
  Let $x$ be a center of $M$. By subdividing the edges in $E$, we can
  assume that $E$ consists of $n_0$-edges, where $n_0\in \N_0$ is large enough. For  $n\geq n_0$ let
  \begin{equation*}
    E^n\coloneqq  \bigcup\{e\in \E^n : e\subset E, x\in e\}.
  \end{equation*}
  Clearly, each set $E^n$ is compact and connected.
  We claim that these  sets form a decreasing sequence, i.e., $E^{n+1}\sub E^n$  for $n\ge n_0$. To see this, suppose $e$ is one of the $(n+1)$-edges forming the union $E^{n+1}$, where $n\ge n_0$. 
   Then $x\in e\sub E$. In particular, $e$ is contained in the union of the $n_0$-edges forming $E$. By subdividing these edges into $n$-edges, we see that $e$ is covered  by  $n$-edges contained in $E$. This implies that  $e$ is contained in one of these $n$-edges $e'$ (this again  follows from  considerations as in the first part of the proof of  Lemma~\ref{lem:mincell}). Hence   $x\in e\sub e'\sub E$, and so $e\sub e'\sub E^n$. Since the $(n+1)$-edge $e\sub E^{n+1}$  was arbitrary,  we conclude  $E^{n+1}\sub E^n$ as desired. 
   
The set $C\coloneqq \bigcap_{n\ge n_0}E^n \sub E$ is an intersection of a decreasing sequence of compact and connected sets, and so it is also compact and connected.
 
 We claim that $C= E \cap M$. Indeed, $E^n\sub \overline \Om^n$, where  
  $\Om^n=\Om^n(x)$  and $n\ge n_0$,  
 as follows from \eqref{eq:closure_Omn}. Hence 
 $C\sub E\cap \bigcap_{n\ge n_0}\overline \Om^n=E\cap M$ by 
 Lemma~\ref{lem:eclass}.

 For the other inclusion, let  $y\in E \cap M$ be arbitrary. Then for each $n\ge n_0$ the point $y$ is contained in an $n$-edge $e\sub E$. Since $y\sim x$, we have $e\cap M\ne \emptyset$, and so $x\in e$ by Lemma~\ref{lem:eq_intersections}. Hence $e\sub E^n$ and so $y\in e\sub E^n$. It follows that $y\in  \bigcap_{n\ge n_0}E^n=C$. This shows the other inclusion $E\cap M\sub C$. We conclude  that $E\cap M=C$ is connected. 
  \end{proof}


For the proof of the next lemma we need the $1$-di\-men\-sional version of Moore's theorem as provided by Proposition~\ref{lem:1dimMoore}. 

\begin{lemma}
  \label{lem:qcells} 
  Let $\tau$ be an edge or a tile. Then $\widetilde \tau$ is an arc or
  a closed Jordan region, respectively. Moreover, $\partial \widetilde
  \tau= \widetilde{\partial \tau}$.
\end{lemma}
 
Here $\partial \tau $ (and similarly $\partial \widetilde \tau$)
refers as usual to the boundary of a cell $\tau$ as defined in
Section~\ref{s:celldecomp}. So $\partial \tau $ is the topological
boundary of $\tau$ in $S^2$ if $\tau$ is a tile, and equal to the set
consisting of the two endpoints of $\tau$ if $\tau$ is an edge. If
$\tau$ is a $0$-dimensional cell, i.e., a singleton set consisting of a
vertex, then $\partial \tau=\emptyset$, and the statement in the lemma
is trivially also true.  So the lemma can be formulated in an equivalent
form by saying that if $\tau\sub S^2$ is a cell (in one of the cell
decompositions $\DD^n$), then $\widetilde \tau\sub \widetilde S^2$ is a
cell (in the general topological sense) of the same dimension, and the
boundary of $\widetilde \tau$ is the image of the boundary of $\tau$
under the quotient map.
 
\begin{proof} 
  Suppose first that $\tau$ is an edge.  Then our equivalence relation
  $\sim$ on $S^2$ restricts to an equivalence relation on $\tau$ whose
  quotient space can be identified with the subset $\widetilde \tau$
  of $\widetilde S^2$.  The equivalence classes on $\tau$ have the
  form $\tau\cap M$, where $M\sub S^2$ is an equivalence class with
  respect to $\sim$.
 
  Each set $\tau \cap M$ is compact, as $\sim$ is closed, and
  con\-nected by Lem\-ma~\ref{lem:edgeinter}. Moreover, $\tau$ meets at
  least two distinct equivalence classes, as its endpoints are
  distinct vertices and hence not
  equivalent. Proposition~\ref{lem:1dimMoore} implies that $\widetilde
  \tau\sub \widetilde S^2$ is indeed an arc.

  Let $u$ and $v$ be the two endpoints of $\tau$.  Then $[u]\cap \tau$
  is a compact connected subset of $\tau$ containing $u$. Hence this
  set is a subarc of $\tau$ with one endpoint equal to $u$. This
  implies that the set $\tau\setminus [u]$ is connected, and so the
  set $\pi(\tau\setminus [u])=\widetilde \tau\setminus \{\pi(u)\}$ is
  also connected.  Therefore, $\pi(u)$ is an endpoint of $\widetilde
  \tau$. By the same reasoning we see that
  $\pi(v)$ is also an endpoint of $\widetilde \tau$. Since $u$ and $v$
  are distinct vertices, we have $u\not\sim v$ and so $\pi(u)\ne
  \pi(v)$. Hence $\partial \widetilde \tau =\{\pi(u), \pi(v)\}=
  \pi(\{u,v\})= \widetilde {\partial \tau}$.

If $\tau$ is a tile, say an $n$-tile, then $\tau$ is a closed Jordan region whose boundary  $J=\partial  \tau$ is a topological circle consisting of finitely many edges. By the Sch\"{o}nflies theorem we can write $ 
S^2$ as a disjoint union 
$ S^2=U_1\cup  J \cup U_2$, 
where  $U_1$ and $U_2$ are open Jordan regions bounded by $J$.  Then  $\tau$ coincides with one of the sets $\overline U_1$ or $\overline U_2$, say $\tau=\overline U_1$.

The set $\widetilde J\sub \widetilde S^2$ is also a topological circle as follows from the fact that $\sim$ is closed, Lemma~\ref{lem:edgeinter}, and  Proposition~\ref{lem:1dimMoore}. So we can also write 
$\widetilde S^2$ as a disjoint union 
$\widetilde 
S^2=D_1\cup \widetilde J \cup D_2$, where $D_1$ and $D_2$ are open Jordan regions in $\widetilde S^2$ bounded by $\widetilde J$.
If we take preimages under the quotient map $\pi\:S^2\ra \widetilde S^2$, we get the disjoint union
$S^2=\pi^{-1}(D_1)\cup \pi^{-1}(\widetilde J)\cup 
\pi^{-1}(D_2)$. 
Recall from Lemma~\ref{lem:simusc} that $\sim$ is monotone,
meaning that each equivalence class is connected. 
Therefore, preimages of connected sets under $\pi$ 
are connected (see Lemma~\ref{lem:pre_monotone}).
So the sets 
$\pi^{-1}(D_1)$ and $\pi^{-1} (D_2)$ are  connected open sets disjoint from $ \pi^{-1}(\widetilde J)\supset J$.  It follows that each of the  sets $\pi^{-1}(D_1)$ and $\pi^{-1} (D_2)$  is contained in one of the 
regions $U_1$ and $U_2$. 

These sets  cannot be contained in the same region $U_i$. 
Indeed, if for example $\pi^{-1}(D_1)\cup \pi^{-1} (D_2)\sub
U_1$,
 then $U_2 \sub \pi^{-1}(\widetilde J)$, and so $\pi(U_2)\sub \widetilde J$. This means that every point in $U_2$ is equivalent to a point 
in $J$. This is impossible, because $U_2$ contains the interior
of an $n$-tile, and hence a $k$-vertex for some $k>n$ (see
Lemma~\ref{lem:vinint}~\ref{item:vinint2}). Such a vertex is not
equivalent to any point in $J$ by condition \ref{item:erel3} in  Lemma~\ref{lem:erel}. 

By what we have just  seen, we may assume that indices are chosen   such that $\pi^{-1}(D_1)\sub U_1$ and 
$\pi^{-1}(D_2)\sub U_2$. Since $\pi$ is surjective, it follows that 
 $$ \overline D_1=D_1\cup \widetilde J= \pi(\pi^{-1}(D_1))\cup \pi(J)\sub 
 \pi (U_1\cup J)=\pi(\overline U_1).$$ On the other hand,
 $\overline U_1\cap \pi^{-1}(D_2)\sub  \overline U_1 \cap U_2=\emptyset$, and so no point in $\overline U_1$ is sent to $D_2=\widetilde S^2\setminus  \overline D_1$ by $\pi$. Hence 
 $\pi(\overline U_1)\sub  \overline D_1$. 
 It follows that  $\overline 
D_1=\pi(\overline  U_1)= \widetilde \tau$, and so $\widetilde \tau$ is indeed  a closed Jordan region.
Moreover, $\partial \widetilde \tau=\partial D_1=\widetilde J=
\widetilde{\partial \tau}$. 
\end{proof}

We now come to the main result of this subsection, which says
that $\pi$ maps the cell decompositions $\DD^n$ to cell
decompositions $\widetilde{\DD}^n$ such that ``all combinatorial
properties are preserved''.

\begin{lemma} 
  \label{lem:cdcd} 
  Let $n,k\in \N_0$. 
  Then the following statements are true:
  \begin{enumerate}
  \item
    \label{item:cdcd1} 
    For each $\tau\in \DD^n$ the set $\widetilde \tau$ is a
    topological cell in $\widetilde S^2$ of the same dimension as
    $\tau$. 
  \item 
    \label{item:cdcd1a}
    For each cell $\tau\in \DD^n$ we have $\widetilde{\partial \tau}
    = \partial \widetilde{\tau}$. 
  \item
    \label{item:cdcd2} 
    For  $\sigma, \tau\in \DD^n$, we have  $\widetilde
    \sigma=\widetilde \tau$ if and only if $\sigma=\tau$.  
  \item
    \label{item:cdcd3} 
    $\widetilde \DD^n\coloneqq \{\widetilde \tau: \tau \in \DD^n\}$ is a cell
    decomposition of $\widetilde S^2$.  
  \item
    \label{item:cdcd4} 
    The map $\tau \in \DD^n\mapsto  \widetilde \tau\in \widetilde \DD^n$
    is an isomorphism between the cell complexes $\DD^n$ and $\widetilde
    \DD^n$.  
  \item
    \label{item:cdcd5} 
    $\widetilde \DD^{n+k}$  is a refinement of $\widetilde
    \DD^n$. Moreover,  for all $\sigma\in \DD^{n+k}$ and $\tau\in
    \DD^{n}$ we have    
      $\sigma\subset \tau$ 
      if and only if 
     $ \widetilde{\sigma}\subset \widetilde{\tau}$.
  \end{enumerate}
\end{lemma}

\begin{proof} 
  \ref{item:cdcd1} and \ref{item:cdcd1a} follow from
  Lemma~\ref{lem:qcells}.  

\smallskip
\ref{item:cdcd2}
\mbox{}
Let $\sigma$ and $\tau$ be arbitrary 
$n$-cells, and suppose that $\inte(\widetilde \sigma)\cap
\inte(\widetilde \tau)\ne \emptyset$. Pick a point $p\in
\inte(\widetilde \sigma)\cap \inte(\widetilde \tau)$. Then $p\in
\widetilde \sigma \cap \widetilde \tau=\widetilde {\sigma \cap \tau}$
(see Lemma~\ref{lem:nonewint}), and so 
there exists $x\in \sigma \cap \tau$ with $\pi(x)=p$. 
Then $x\in \inte(\sigma)$, for otherwise 
$x\in \partial \sigma$ and so $p=\pi(x)\in \partial 
\widetilde \sigma$ by \ref{item:cdcd1a},  
contradicting  the choice of $p$. Similarly, $x\in \inte(\tau)$. So $x\in \inte(\sigma)\cap \inte(\tau)$ which implies that $\sigma=\tau$. Statement \ref{item:cdcd2} follows. 

\smallskip{}
\ref{item:cdcd3}
From what we have seen, it follows that the topological cells
$\widetilde \tau$ for $\tau \in \DD^n$ are all distinct, and no two
have a common interior point. Moreover, there are finitely many of
these cells, and they cover $\widetilde S^2$, because the cells in
$\DD^n$ cover $S^2$. Finally, for a cell $\widetilde \tau$ we have
$\partial \widetilde \tau=\widetilde {\partial \tau}$ by
\ref{item:cdcd1a}. Since $\partial \tau$ is a union of cells in
$\DD^n$, the set  $\partial \widetilde \tau$ is a union of
cells in $\widetilde \DD^n$.  This shows that $\widetilde \DD^n$ is a
cell decomposition of $\widetilde S^2$. 

\smallskip
\ref{item:cdcd4} By \ref{item:cdcd1} and \ref{item:cdcd2} the map $\tau \in \DD^n\mapsto  \widetilde \tau\in \widetilde \DD^n$ is a bijection between $\DD^n$ and $\widetilde \DD^n$ that preserves dimensions of cells. By Lemma~\ref{lem:nonewint}  the map also satisfies condition \ref{item:compiso2} 
in Definition~\ref{def:compiso}. Hence it  is an isomorphism between
the cell complexes $\DD^n$ and $\widetilde \DD^n$. 

\smallskip 
\ref{item:cdcd5} It follows immediately from the definitions and the
fact that $\DD^{n+k}$ is a refinement of $\DD^n$ that
$\widetilde{\DD}^{n+k}$ is a refinement of $\widetilde{\DD}^{n}$. The
second statement was proved in Lemma~\ref{lem:nonewint}.  
\end{proof}

\subsection*{The induced map $\widetilde f$ on $\widetilde S^2$}
\label{sec:induc-map-widet}
We will now show that $f$ induces a map $\widetilde f$ on the
sphere $\widetilde S^2$. 

\begin{lemma} 
  \label{lem:Find} 
  The equivalence relation $\sim$ is $f$-invariant. 
\end{lemma} 

\begin{proof} Let $x,y\in S^2$ with $x\sim y$ be arbitrary. We have to show that 
then $f(x)\sim f(y)$ (see \eqref{eq:eq_f_inv}).  

 Pick $\{X^n\}, \{Y^n\}\in \mathcal{S}$ with $x\in \bigcap X^n$ and $y\in \bigcap_n Y^n$.  Define $U^n=f(X^{n+1})$  and $V^n=f(Y^{n+1})$ for $n\in \N_0$. Then $U^n$ and $V^n$ are $n$-tiles, and so 
$\{U^n\}, \{V^n\}\in \mathcal{S}$. Moreover,  $f(x)\in \bigcap_n U^n$ 
and $f(y)\in \bigcap_n V^n$. Since $x\sim y$ we have $X^n\cap Y^n\ne\emptyset$ 
for all $n\in \N$. Hence
$$U^n\cap V^n=f(X^{n+1})\cap f(Y^{n+1})\supset f(X^{n+1}\cap Y^{n+1})\ne \emptyset$$
for all $n\in \N_0$. Lemma~\ref{lem:erel} now implies that $f(x)\sim f(y)$ as desired.
\end{proof}
 
By the previous lemma the map $\widetilde f\: \widetilde S^2\ra
\widetilde S^2$ given by
$$ \widetilde f([x])=[f(x)]  \text{ for  $x\in S^2$}$$ 
is well-defined. Then $ \widetilde f\circ \pi=\pi\circ f$, and it
follows from the properties of the quotient topology that $\widetilde
f$ is continuous (see Lemma~\ref{lem:f_descends}).

\begin{rem}\label{rem:stronginv}
With some  additional effort, one can actually establish  that $\sim$ is strongly $f$-invariant (see Definition~\ref{def:sim_strongly-inv}). For this one first shows 
(using Lemmas ~\ref{lem:erel} and~\ref{lem:eq_intersections}) that each 
equivalence class $M$ of $\sim$ can be represented in the form 
  \begin{equation}\label{eq:eqclrep}
    M= \bigcup \bigg\{\bigcap_n X^n:  
      \{X^n\} \in \mathcal{S}, c\in \bigcap_n X^n\bigg\},\end{equation}
      where $c$ is a center of $M$. For given $x\in S^2$ one then chooses a center  $c$ of 
      $M=[f(x)]$  and shows that it has a preimage $c'\in [x]$ under $f$. By analyzing the different types for $M$ and  invoking \eqref{eq:eqclrep} in combination with   Lemma~\ref{lem:eclass}, one can then prove  that $f([x])=[f(x)]$. This  implies that $\sim$ is indeed   strongly $f$-invariant.

{}From this  one can conclude that 
 $\widetilde f$ is a Thurston map based on 
 Corollary~\ref{cor:f_descends_Thurston}. We will actually provide a direct simple  argument for this that will also show that $f$ and $\widetilde{f}$ are  Thurston equivalent.  \end{rem}

In the following, $\widetilde \DD^n=\{\widetilde \tau: \tau \in
\DD^n\}$ for $n\in \N_0$ will denote the cell decomposition of $\widetilde S^2$ as
provided by Lemma~\ref{lem:cdcd}~\ref{item:cdcd3}. 
As the next lemma shows, the map ${\widetilde f}^n$ has injectivity
properties  similar to  $ f^n$.  

\begin{lemma}\label{lem:cellinj}
Let $\tau$ be an $n$-cell, $n\in \N$. Then 
${\widetilde f}^n$ is a homeomorphism of $\widetilde \tau$ onto $\widetilde \sigma$, where $\sigma=f^n(\tau)$.  In particular,  $\widetilde f^n$ is cellular for $(\widetilde \DD^n, \widetilde \DD^0)$.
\end{lemma}

\begin{proof} Since ${\widetilde f}$ is continuous,  $\widetilde f^n$ is also continuous. 
 Note that $f^n$ is a homeomorphism of $\tau$ onto $\sigma$. Hence    
$$\widetilde f^n(\widetilde \tau)= (\widetilde f^n\circ \pi)(\tau)=(\pi\circ f^n)(\tau)=\widetilde \sigma$$ 
showing that ${\widetilde f}^n$ maps $\widetilde \tau$ onto $\widetilde \sigma$.

So it remains to show the injectivity of ${\widetilde f}^n$ on
$\widetilde \tau$,  or equivalently, that if $x,y\in \tau$ and
$f^n(x)\sim f^n(y)$, then $x\sim y$.  Since every $n$-vertex 
and every $n$-edge is contained in an $n$-tile, we may also assume that $\tau$ is an $n$-tile. 

If   $x,y\in \tau$, then  
we can pick sequences $\{X^k\}$ and $\{Y^k\}$ in $ \mathcal{S}$ such that $X^n=Y^n=\tau$ and $x\in \bigcap_kX^k$, $y\in \bigcap_k Y^k$. 
Then $f^n(X^{k+n})$ and $f^n(Y^{k+n})$ are $k$-tiles for $k\in \N_0$. Moreover, the sequences  $\{f^n(X^{k+n})\}$ and 
$\{f^n(Y^{k+n})\}$ are in $\mathcal{S}$, and 
$f^n(x)\in  \bigcap_k f^n(X^{k+n})$ and $f^n(y)\in  \bigcap_k f^n(Y^{k+n})$.  Since $f^n(x)\sim f^n(y)$, 
this implies that $f^n(X^{k+n})\cap f^n(Y^{k+n})\ne \emptyset $ for all $k\in \N_0$. Since $X^{k+n}, Y^{k+n}\sub \tau$ for $k\ge 0$ and $f^n|\tau$ is injective, we conclude that $X^{k+n}\cap Y^{k+n}\ne \emptyset$ for $k\ge 0$. Since $X^n=Y^n=\tau$, we also have 
$X^k=Y^k$ for $k=0, \dots , n-1$. Hence $X^k\cap Y^k\ne \emptyset$ for all $k\ge 0$.  Lemma~\ref{lem:erel} then shows that $x\sim y$ as desired. 

 The fact that $\widetilde f^n$ is cellular for $(\widetilde \DD^n, \widetilde \DD^0)$ follows from the first part of the proof and the fact that $f^n$ is cellular for $(\DD^n, \DD^0)$.  
\end{proof}

\subsection*{The auxiliary homeomorphisms $h_0$ and $h_1$}
\label{sec:auxil-home-h_0}
To prove that $\widetilde{f}$ is a Thurston map equivalent to $f$, we need
to define homeomorphisms  
$h_0,h_1\: S^2\ra \widetilde S^2$ that make the  diagram 
\begin{equation}
  \xymatrix{
    S^2 \ar[r]^{h_1} \ar[d]_f & \widetilde S^2 \ar[d]^{\widetilde f}
    \\
    S^2 \ar[r]^{h_0} & \widetilde S^2
  }
\end{equation}
commutative and are isotopic rel. $\V^0=\post(f)$.  The construction
of these maps follows ideas in the proof of
Lemma~\ref{lem:isotwotequiv}.
  
For the definition of $h_0$ recall that $S^2$ is the union of two
$0$-tiles $X^0_{\texttt{b}}$ and $X^0_{\texttt{w}}$ with common
boundary $\mathcal{C}$.  The Jordan curve $\mathcal{C}$ is further
decomposed into $k=\#\V^0\ge 3$ $0$-edges and $0$-vertices.  
The cell decomposition $\widetilde \DD^0$ of $\widetilde S^2$ contains
two tiles $\widetilde X^0_{\texttt{b}}$ and $\widetilde
X^0_{\wt}$. Lemma~\ref{lem:cdcd}~\ref{item:cdcd1a} and
Lemma~\ref{lem:qcells} show that the common boundary of
$\widetilde{X}^0_{\bt}$ and $\widetilde{X}^0_\wt$ is $\widetilde
\CC=\pi(\CC)$, which is a Jordan curve. There are $k$ distinct vertices
and edges on $\widetilde \CC$. There are no other cells in $\widetilde
\DD^0$.

We know by Lemma~\ref{lem:cdcd}~\ref{item:cdcd4} that the map $\tau\in \DD^0\mapsto \widetilde \tau\in \widetilde \DD^0$ is an isomorphism between the cell complexes $\DD^0$ and $\widetilde \DD^0$. So 
Lemma~\ref{lem:isocellhomeo}~\ref{item:isocellhomeo2} implies that there exists a homeomorphism $h_0\: S^2\ra \widetilde S^2$ such that $h_0(\tau)=\widetilde \tau$ for all cells $\tau\in \DD^0$. 

Now let  $\tau\in \DD^1$ be arbitrary. Then $f(\tau)\in \DD^0$, and by Lemma~\ref{lem:cellinj}
the map $\widetilde f|\widetilde \tau$ is a homeomorphism of $\widetilde \tau$  onto $\widetilde {f(\tau)}={h_0(f(\tau))}$. Hence the map
$$\varphi_{\tau}\coloneqq (\widetilde f|\widetilde \tau)^{-1}\circ h_0 \circ (f|\tau)$$ is well-defined and a homeomorphism from $\tau$ onto $\widetilde \tau$.
If $x\in \tau$, then $y=\varphi_\tau(x)$ is the unique point $y\in \widetilde \tau $ with $\widetilde f(y)=h_0(f(x))$. As in the proof of Lemma~\ref{lem:isotwotequiv}, this uniqueness property implies that if  $\sigma,\tau\in \DD^1$ and   $\sigma\sub \tau$, then 
$ \varphi_\tau|\sigma=\varphi_\sigma. $ From this in turn one can deduce that 
if  a point $x\in S^2$ lies in two cells $\tau,\tau'\in \DD^1$, then
$\varphi_\tau(x)=\varphi_{\tau'}(x)$. 
This allows us to define a map $h_1\: S^2\ra \widetilde S^2$ as follows. If $x\in S^2$, we pick   $\tau\in \DD^1$ with $x\in \tau$ and set 
$h_1(x)\coloneqq \varphi_\tau(x).$ Then $h_1\:S^2 \ra \widetilde S^2$ is well-defined.

\begin{lemma}\label{lem:phi}
The map $h_1\: S^2\ra\widetilde S^2$ is a
 homeomorphism of $S^2$ onto $\widetilde S^2$ satisfying $h_0\circ f=\widetilde f\circ h_1$. Moreover, we have $h_0(\CC)=\widetilde \CC=h_1(\CC)$ and the homeomorphisms 
$h_0$ and $h_1$ are isotopic rel.\ $\V^0=\post(f)$. 
\end{lemma} 

\begin{proof} We have $h_1|\tau=\varphi_\tau$ for each cell $\tau\in \DD^1$. So the  definitions of $h_1$ and $\varphi_\tau$ show that $h_0\circ f=\widetilde f\circ h_1$ and that 
$h_1(\tau)=\varphi_\tau(\tau)=\widetilde \tau$ for each $\tau\in \DD^1$. Since $\tau\in \DD^1\mapsto \widetilde \tau \in \widetilde 
\DD^1$ is an isomorphism of cell complexes by 
Lemma~\ref{lem:cdcd}~\ref{item:cdcd4}, the last statement implies that $h_1$ is a homeomorphism of $S^2$ onto $\widetilde S^2$ (Lemma~\ref{lem:isocellhomeo}~\ref{item:isocellhomeo1}). 

Note that $h_1(\tau)=\widetilde \tau$ also for each $\tau\in \DD^0$. Indeed, suppose  
 $\tau\in \DD^0$ is   arbitrary. Since $\DD^1$ is a refinement of $\DD^0$, for each $x\in \tau$ there exists $\sigma\in \DD^1$ such that $x\in \sigma\sub \tau$. Then 
 $h_1(x)\in h_1(\sigma)=\widetilde \sigma\sub \widetilde \tau$. So $h_1(\tau)\sub \widetilde \tau$. Conversely, let  $y\in \widetilde \tau$ be arbitrary. Since $\widetilde \DD^1$ is a refinement of $\widetilde \DD^0$ 
 (Lemma~\ref{lem:cdcd}~\ref{item:cdcd5}), there exists a cell $\sigma\in \DD^1$ such that 
 $y\in \widetilde \sigma\sub \widetilde \tau$. By Lemma~\ref{lem:nonewint} we then have $\sigma\sub \tau$, and so $y\in \widetilde \sigma=h_1(\sigma)\sub h_1(\tau)$. We conclude that 
 $h_1(\tau)=\widetilde \tau$ for each $\tau\in \DD^0$ as claimed.   

The Jordan curve $\mathcal{C}$ is  the $1$-skeleton of $\DD^0$ and thus equal to the  union of all edges  $e\in \DD^0$. We know that  
$h_0(e)=h_1(e)=\widetilde e=\pi(e)$ for each such edge $e$. Hence 
  $h_0(\mathcal{C}) =h_1(\CC)=\pi(\mathcal{C})=
\widetilde {\mathcal{C}}$.

If $\tau\in \DD^0$, then $h_0(\tau)= \widetilde \tau =h_1(\tau)$.
  So by
Lemma~\ref{lem:isocellhomeo}~\ref{item:isocellhomeo3} (applied to the isomorphism $\tau \in \DD^0 \mapsto \widetilde \tau \in \widetilde {\DD}^0$) the
homeomorphisms $h_0$ and $h_1$ are isotopic rel.\ $\V^0=\post(f)$.
\end{proof}

\begin{lemma}
  \label{lem:widefThurston}
  The map $\widetilde f \: \widetilde S^2 \ra \widetilde S^2$ is
  a Thurston map. It is Thurston equivalent to $f$ and satisfies
  $\post (\widetilde f)=\pi(\post(f))$. Moreover, if
  $\widetilde \CC=\pi(\CC)\sub \widetilde S^2$, then
  $\widetilde \CC$ is an $\widetilde f$-invariant Jordan curve
  with $\post(\widetilde f)\sub \widetilde \CC$.
\end{lemma} 

\begin{proof} 
  As we have 
  already
seen in Lemma~\ref{lem:phi}, there exist 
 homeo\-mor\-phisms $h_0, h_1\: S^2\ra \widetilde S^2$ that are isotopic rel.\ $\post(f)$ and satisfy 
 $h_0\circ f= \widetilde f\circ h_1$. Then $\widetilde f$ is a Thurston map 
 with  $\post(\widetilde{f})= h_0(\post(f))= \pi(\post(f))$  by Lemma~\ref{lem:T-eq_crit_post}, and it is clear that $f$ and $\widetilde{f}$ are  Thurston equivalent.
 
We know that   $\widetilde {\mathcal{C}}=\pi(\mathcal{C})\sub 
\widetilde S^2$ is a Jordan curve. It  satisfies 
$$\post (\widetilde f)=\pi(\post(f))\sub \pi(\mathcal{C})=\widetilde {\mathcal{C}}. $$
Since $f(\mathcal C)\sub \mathcal C$, we also have  
$$ \widetilde f(\widetilde {\mathcal{C}})=
 (\widetilde f\circ \pi)(\mathcal{C})=(\pi\circ f)
 (\mathcal{C})\sub \pi(\mathcal{C})=
 \widetilde {\mathcal{C}}. $$
 This shows that $\widetilde {\mathcal{C}}$ is 
 $ \widetilde f$-invariant 
   and contains the set of postcritical points of 
 $ \widetilde f$.
 \end{proof}

Let $L\: \DD^1\ra \DD^0$ be the labeling induced by $f$.  It is
given by $L(\tau)=f(\tau)\in \DD^0$ for $\tau\in \DD^1$ (see
Section~\ref{sec:labelings}). Since the Jordan curve
$\CC\subset S^2$ with $\post(f) \subset\CC$ that was used to
define $\DD^0=\DD^0(f,\CC)$ and $\DD^1=\DD^1(f,\CC)$ is
$f$-invariant,  $(\DD^1, \DD^0, L)$ is a two-tile subdivision
rule realized by $f$ (see Proposition~\ref{prop:ThmapSub}).

We   consider the associated cell decompositions $\widetilde \DD^0$ and $\widetilde \DD^1$ of the $2$-sphere $\widetilde S^2$ as given by Lemma~\ref{lem:cdcd}~\ref{item:cdcd3}. By Lemma~\ref{lem:cdcd}~\ref{item:cdcd4} each cell in $\widetilde \DD^1$ can be represented as $\widetilde \tau$ for a unique $\tau\in \DD^1$. This implies that if we set $\widetilde L(\widetilde \tau)\coloneqq\widetilde {L(\tau)}=\widetilde {f(\tau)}\in \widetilde \DD^0$ for 
$\tau\in \DD^1$, then we obtain a well-defined map   $\widetilde L\: 
 \widetilde \DD^1\ra \widetilde \DD^0$. 
  
\begin{cor}
  \label{cor:isosubdic} 
  \index{two-tile subdivision rule!isomorphism of}
  \index{subdivision!isomorphic}
  \index{isomorphism!of two-tile subdivision rules}
  The map $\widetilde L\: 
 \widetilde \DD^1\ra \widetilde \DD^0$ is a   labeling.  Moreover, 
  $(\widetilde \DD^1, \widetilde \DD^0, \widetilde L)$ is a two-tile subdivision rule isomorphic to $(\DD^1,  \DD^0,  L)$. It is realized by the Thurston map $\widetilde f$. 
\end{cor}  

Essentially, this corollary says that all the combinatorial information encoded in $f$ and its associated two-tile subdivision rule $(\DD^1,  \DD^0,  L)$ is preserved if we pass to the quotient space $\widetilde S^2$.

\begin{proof} It is clear  that $\widetilde \DD^0=\DD^0(\widetilde 
f, \widetilde \CC)$, where, as before, $\widetilde \CC=\pi(\CC)$. Since the map $\widetilde f$ is cellular for $(\widetilde \DD^1, \widetilde \DD^0)$ by Lemma~\ref{lem:cellinj},  the uniqueness statement in  Lemma~\ref{lem:pullback} implies that 
$\widetilde \DD^1=\DD^1( \widetilde f, \widetilde \CC)
$. 

 Note that 
$\widetilde L$ is the labeling induced by  the Thurston map   $\widetilde f$.
Indeed,  
each cell in  $\widetilde \DD^1$ has a 
representation of the form 
$\widetilde \tau$ with a unique $\tau \in \DD^1$. As we have seen in the proof of Lemma~\ref{lem:phi},  we have 
$h_1(\tau)=\widetilde \tau$. Moreover, $f(\tau)\in \DD^0$ and so
 $h_0(f(\tau))=\widetilde {f(\tau)}= \widetilde L(\widetilde \tau)$ by the definitions of $h_0$ and $\widetilde L$. This leads to the  desired relation 
$$ \widetilde f(\widetilde \tau)=(\widetilde f\circ h_1)(\tau)=(h_0\circ f)(\tau)=
h_0(f(\tau))=\widetilde L(\widetilde \tau). $$
 
 Since 
 $\widetilde \CC$ is an $\widetilde f$-invariant Jordan curve with $\post(\widetilde f)\sub \widetilde \CC$, 
 Proposition~\ref{prop:ThmapSub} shows  that $(\widetilde \DD^1, \widetilde \DD^0, \widetilde L) 
 =(\DD^1( \widetilde f, \widetilde \CC), \DD^0( \widetilde f, \widetilde \CC), \widetilde L) $ 
 is a two-tile subdivision rule realized by $\widetilde f$.
 
  Finally, by using the cell complex isomorphisms 
  $\tau \in \DD^i\mapsto \widetilde \tau \in \widetilde \DD^i$ for $i=0,1$  is easy to see that 
  $(\DD^1,  \DD^0,  L)$ and $(\widetilde \DD^1, \widetilde \DD^0, \widetilde L)$ are isomorphic. 
\end{proof}

We are now ready to prove the main results of this chapter. 

\begin{proof}[Proof of Proposition~\ref{prop:combexp}] 
  Let $f$ be a Thurston map that is combinatorially expanding for
  a Jordan curve $\CC$ as in the statement.  
 Then  all of our previous considerations  apply.

We consider the $2$-sphere $\widetilde S^2=S^2/{\sim}$,  the 
 quotient map $\pi\: S^2\ra  \widetilde S^2$,  the Thurston   map $\widetilde f\: \widetilde S^2\ra \widetilde S^2$, the Jordan curve $\widetilde \CC=\pi(\CC)$, and the homeomorphisms $h_0$ and $h_1$ defined earlier. 
   Then it  follows from Lemmas~\ref{lem:phi} and~\ref{lem:widefThurston} 
 that we have all the desired properties, but it remains to show that 
 $\widetilde f$ is expanding.  Since $\widetilde \CC
\sub \widetilde S^2$ is an $\widetilde f$-invariant Jordan curve with 
$\post (\widetilde f)\sub \widetilde \CC$, we can do this by verifying
the condition in Lemma~\ref{lem:charexpint} for $\widetilde\CC$. 

We have  $\DD^0(\widetilde f, \widetilde \CC)=\widetilde \DD^0$. Moreover, since the map $\widetilde f^n$ is cellular for $(\widetilde \DD^n, \widetilde \DD^0)$,  it follows from the uniqueness statement in  Lemma~\ref{lem:pullback} that $\DD^n( \widetilde f, \widetilde \CC)=
\widetilde \DD^n$ for all  $n\in \N_0$.   
So  the $n$-tiles for  $(\widetilde f,\widetilde \CC)$ are precisely 
 the sets $\widetilde X=\pi(X)$, where $X$ is an $n$-tile 
 on $S^2$ for $(f,\CC)$. 

Let $\widetilde{X}^0 \supset \widetilde{X}^1 \supset \widetilde{X}^2
\supset \dots$ be a nested sequence of $n$-tiles for $(\widetilde{f},
\widetilde{\CC})$. Clearly, 
$\bigcap_n\widetilde{X}^n$ is non-empty. We have to show that this intersection  does
not contain more than one point.  From
Lemma~\ref{lem:cdcd}~\ref{item:cdcd5} it follows that the
corresponding sequence $\{X^n\}$ of $n$-tiles for $(f,
\CC)$ is nested, and so $\{X^n\}\in \mathcal{S}=\mathcal{S}(f,\CC)$.  
 To see that $\bigcap_n\widetilde X^n$ consists of precisely one point, we argue by contradiction and assume that  $\bigcap_n \widetilde X^n$ contains more than one point, or equivalently, that there exist two distinct (and hence disjoint) equivalence classes $M$ and $N$ with respect to $\sim$ such that $M^n\coloneqq M\cap X^n\ne \emptyset$ and $N^n\coloneqq N\cap X^n \ne \emptyset$ for all $n\in \N_0$.  
 Since equivalence classes and tiles are compact, in this way we get descending sequences $M^0\supset M^1\supset \dots$ and $N^0\supset N^1\supset\dots$ of non-empty and compact sets. Hence 
 the sets $\bigcap_n M^n=M\cap \bigcap_n X^n$ and 
 $\bigcap_n N^n=N\cap \bigcap_n X^n$ are non-empty. So there exist points 
 $x\in M\cap \bigcap_n X^n$ and 
 $y\in N\cap \bigcap_n X^n$. Since $x$ and $y$ lie in different equivalence classes, they are not equivalent. On the other hand, we have $x,y\in 
 \bigcap_n X^n$ and $\{X^n\}\in \mathcal{S}$. Hence $x\sim y$ by Lemma~\ref{lem:erel}. This is a contradiction and we conclude that $\widetilde f$ is indeed expanding.
\end{proof} 


Our previous considerations immediately give the proofs of Theorems~\ref{thm:combexp1} and ~\ref{thm:combexp2}.

\begin{proof}
  [Proof of Theorem~\ref{thm:combexp1}] 
 We  use the same notation as in Pro\-po\-si\-tion~\ref{prop:combexp} 
  and define a homeomorphism 
  $\phi=h_1^{-1}\circ h_0$.  Then $\phi(\CC)=\CC$, 
  and $\phi$ is isotopic to the identity
  on $S^2$ rel.~$\post(f)$. This implies that $\phi$ is orientation-preserving. Moreover,
  $$g=\phi\circ f=h_1^{-1}\circ h_0 \circ f =h_1^{-1} \circ \widetilde f \circ h_1, $$ and so $g$ is
  topologically conjugate to the expanding Thurston map $\widetilde
  f$, and hence itself an expanding Thurston map.

  Similarly, the map 
  \begin{equation*}
    \widetilde{g} = f\circ \phi = f\circ h_1^{-1} \circ h_0 =
    h_0^{-1} \circ \widetilde{f} \circ h_0. 
  \end{equation*}
  is topologically
 conjugate to $\widetilde{f}$, and hence an expanding Thurston map. 
\end{proof}

\begin{proof}[Proof of Theorem~\ref{thm:combexp2}] 
  Let $(\DD^1, \DD^0, L)$ be a two-tile subdivision rule on $S^2$ as in the statement, $\CC$ be the Jordan curve  and ${\bf V}^0$ be the vertex set of $\DD^0$. 
  
  If the subdivision rule $(\DD^1,\DD^0,L)$ 
  can be realized by an expanding Thurston map $f\: S^2\ra S^2$, then  $\CC$ is $f$-invariant,  $\post(f)\sub \CC$, and $\#\post(f)\ge 3$. So Lemma~\ref{lem:Dtoinfty} implies that 
  $f$ is 
  combinatorially expanding for $\CC$, and so $(\DD^1, \DD^0, L)$ is combinatorially expanding 
  according to  Definition~\ref{def:combexprule} and the discussion following this definition. 

  Conversely, suppose $(\DD^1, \DD^0,L)$ is combinatorially expanding. Then this subdivision rule can be realized 
 by a  Thurston map $f\colon S^2\to S^2$ that is combinatorially
 expanding for $\CC$. From our assumptions 
 it follows  that $\post(f)={\bf V}^0$ (see Remark~\ref{rem:two-tile_fV0}~(i)).  Note that then $\DD^0=\DD^0(f,\CC)$. Moreover, 
 $f$ is cellular for $(\DD^1,\DD^0)$ and so necessarily
 $\DD^1=\DD^1(f,\CC)$ (see Lemma~\ref{lem:pullback}).  
 
 We again use the notation of
 Proposition~\ref{prop:combexp}.
We define $\phi= h_1^{-1} \circ h_0$. Then $\phi$ is a
orientation-preserving  homeomorphism on $S^2$ that is isotopic to $\id_{S^2}$  rel.\
 $\V^0=\post(f)$.  As in the proof of
 Theorem~\ref{thm:combexp1}, let $g=\phi\circ f = h_1^{-1} \circ
 \widetilde{f} \circ h_1$. Then $g\colon S^2\to S^2$ is an expanding
 Thurston map.

  Since $\phi(\CC)=\CC$ and $\phi$ is orientation-preserving and the identity on
  $\V^0$, we have $\phi(c)=c$ for each $c\in \DD^0$. Since $f$ is
  cellular for $(\DD^1,\DD^0)$, the map $g=\phi 
  \circ f$ is cellular for $(\DD^1, \DD^0)$ and we have $g(c)=f(c)$
  for each cell $c\in \DD^1$. Since $f$ realizes the given 
  two-tile subdivision rule,  this shows that $g$ is also a
  realization. Now $g$ is expanding, and so the claim follows.
\end{proof}   

We end this chapter with two examples. The first one illustrates why we need 
the condition $\post(f)= \V^0$
in Theorem~\ref{thm:combexp2}. 

\begin{figure}
  \centering
  \begin{overpic}
    [width=11cm, tics=10,
    ]{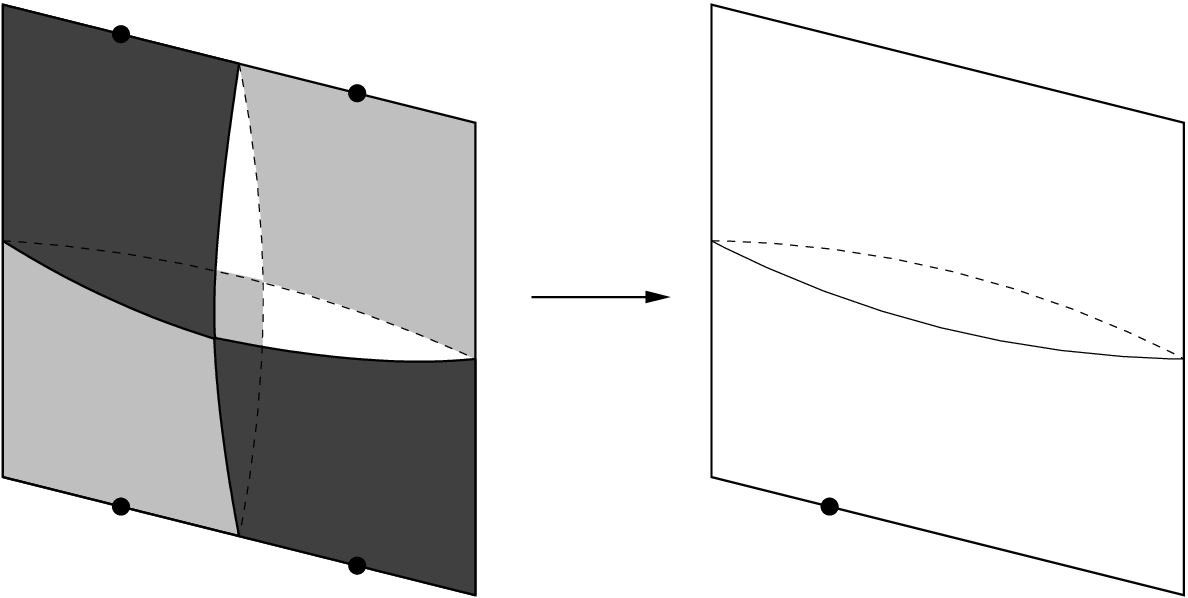}
    %
    \put(20,49){$\DD^1$}
    \put(9,3){$v\mapsto v$}
    \put(-6,6){$0\mapsto 0$}
    %
    \put(80,49){$\DD^0$}
    \put(68,4){$v$}
    \put(59,6){$0$}
    \put(50,27){$g$}
  \end{overpic}
  \caption{A two-tile subdivision rule realized by a map $g$ with $\post(g)\ne \V^0$.}
  \label{fig:extra_vertex_D0}
\end{figure}
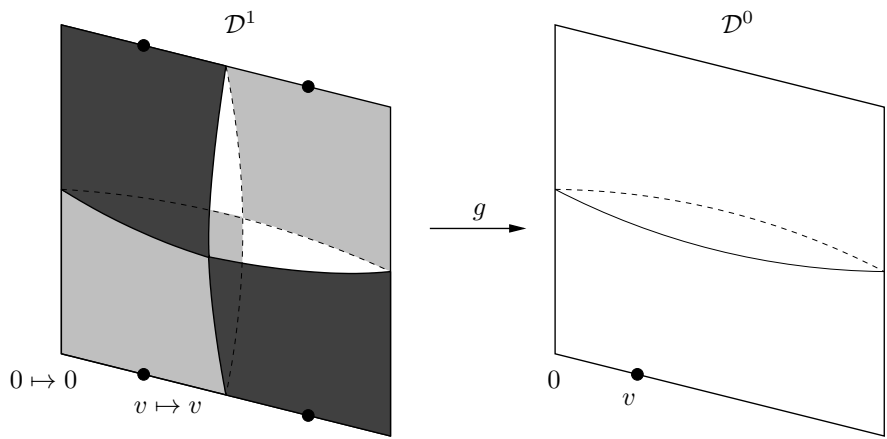

\begin{ex}
  \label{ex:extra_vertex_D0}
  Consider the two-tile subdivision rule $(\DD^1,\DD^0,L)$ indicated 
  in Figure~\ref{fig:extra_vertex_D0}. It is almost the same as
  the one in Figure~\ref{fig:mapg}
   realized by the Latt\`{e}s map discussed in 
  Section~\ref{sec:Lattes}.   However,  there is
  one difference: $\DD^0$ contains an additional vertex $v$; 
  so $\DD^0$ contains five  vertices (and  five  edges).  Compared to Figure~\ref{fig:mapg},
   the cell decomposition $\DD^1$  of the subdivision rule  $(\DD^1,\DD^0,L)$ 
 contains four  additional vertices (represented  by black dots in Figure~\ref{fig:extra_vertex_D0}), one of which agrees
  with $v$. 

  Let $g\colon S^2\to S^2$ be an arbitrary Thurston map that
  realizes this two-tile subdivision rule. Then  the set $\post(g)$ consists of the four corners of the 
  squares forming the pillow, and so $\post(g)\ne \V^0$. Moreover,  $g$ is
  combinatorially expanding for the Jordan curve $\CC$ of $\DD^0$
  (which is the equator of the pillow), because no tile in $\DD^1$ joins opposite sides of 
  $\CC$.  Note that for combinatorial expansion of $g$ for $\CC$ only the points in $\post(g)$ are relevant, and so the extra point $v$ plays no role here.

 Let  $e$ be the edge in  $\DD^0$  whose endpoints are  the vertices labeled by
  $0$ and $v$. Then  $e$ is also an edge in $\DD^1$ and
  $g(e)=e$. This implies that $e$ is  an $n$-edge for $(g,\CC)$ for each $n\in \N_0$
  and so  $g$ cannot be expanding.  
\end{ex}


By  Theorem~\ref{thm:combexp1}  combinatorial
expansion is a sufficient condition for a Thurston map to be
equivalent to an expanding Thurston map. Our last example in this
chapter shows that this condition is not necessary.

\begin{ex}
  \label{ex:exp_notcexp}
 We consider  the map $f\colon S^2\to S^2$  represented by the top part of
  Figure \ref{fig:exp_notcexp}. Here we identify $S^2$ with a
  pillow that is obtained by gluing  two squares together along their boundaries. The
  map $f$ has four postcritical points, which are the vertices of
  the pillow shown on the top right. Its front is the white
  $0$-tile, and its back the black $0$-tile. The subdivision
  of the $0$-tiles is indicated on the top left in  the
  figure. Here we have cut the pillow along three $0$-edges and folded the back of the pillow up so that we see two adjacent squares. The
  left one shows the subdivision of the white $0$-tile, and the
  right one the subdivision of the black $0$-tile. One
  postcritical point (which is a vertex
  of the pillow) is marked by a large black  dot.  On the left we indicated
  its preimages, meaning the $1$-vertices that are labeled by
  this $0$-vertex.  The other $1$-vertices are shown as small
 black  dots.
  
  The equator $\CC$ of the pillow $S^2$  is an $f$-invariant Jordan curve containing
  $\post(f)$. The map $f$ is not combinatorially expanding for
  $\CC$: for each $n\in \N_0$ there is a white
  $n$-tile  (contained in the white $0$-tile) that joins the
  $0$-edges given by  the left and the  right side 
   of the  white $0$-tile.

\ifthenelse{\boolean{nofigures}}{}{
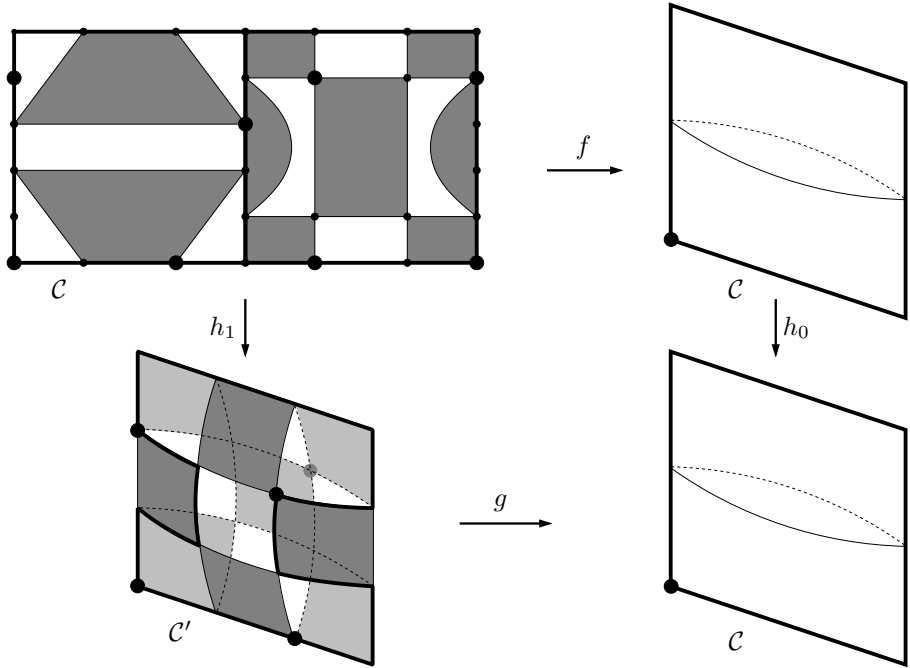
\begin{figure}
  \centering
  \begin{overpic}
    [width=12cm, 
    tics=20]{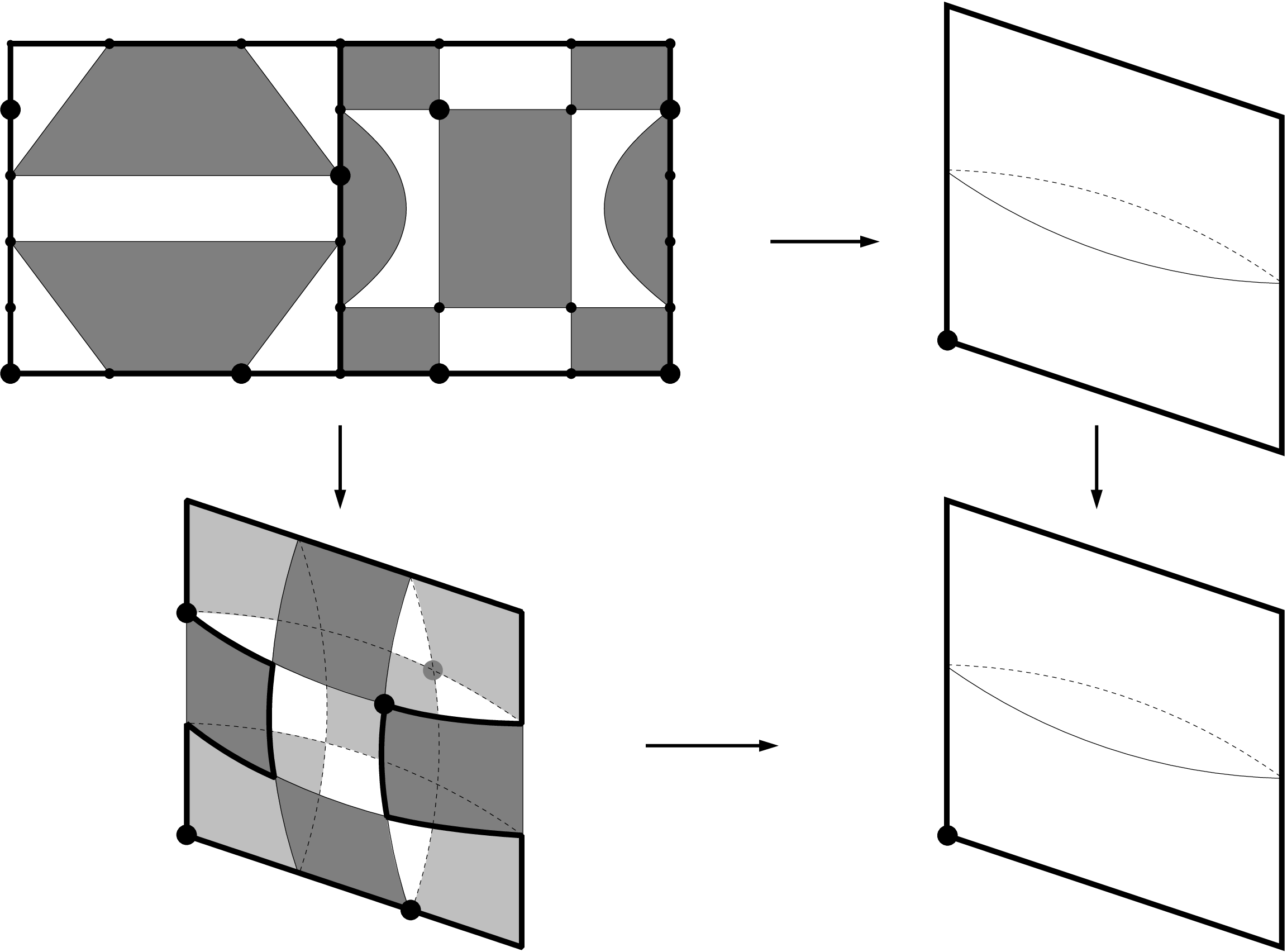}
    \put(63,57){$f$}
    \put(54,18){$g$}
    \put(5,41){$\CC$}
    \put(18,3){$\CC'$}
    \put(80,41){$\CC$}
    \put(22.5,37){$h_1$}
    \put(86,37){$h_0$}
    \put(80,2){$\CC$}
  \end{overpic}
  \caption{The map $f$ is not combinatorially expan\-ding, but equivalent to the expanding map $g$.}
  \label{fig:exp_notcexp}
\end{figure}
}

 We want to show that 
the map
 $f$ is equivalent to an expanding Thurston map
  $g\colon {S}^2\to {S}^2$ defined on the same pillow
  as $f$. The map $g$  is indicated  at   the bottom in 
  Figure~\ref{fig:exp_notcexp}. If we identify the pillow ${S}^2$
  with $\CDach$ in the same way as in Section~\ref{sec:Lattes},
  then $g$ is a Latt\`{e}s map.  It is obtained according to
  Theorem~\ref{thm:Lattesstruc}~\ref{item:Lattessruciii} as a
  quotient of $A \colon \C \to \C$, $z\mapsto A(z)\coloneqq 3z$, by the
  crystallographic group of type $(2222)$  in
  \eqref{eq:specisom}.  Since $g$ is a Latt\`{e}s map, it is
  expanding.

  Let $h_0\coloneqq \id_{S^2}$. We consider 
  $\DD^1= \DD^1(f,\CC)$
  and $\widetilde{\DD}^1= \DD^1(g,{\CC})$ as given by
  Definition~\ref{def:DDn}. These are the cell decompositions of $S^2$
   shown on the left in  Figure~\ref{fig:exp_notcexp}. Note
  that $\DD^1$ and $\widetilde{\DD}^1$ are in fact isomorphic.
 More precisely,  there is a bijection
  $\phi\colon \DD^1 \to \widetilde{\DD}^1$ as in 
  Definition~\ref{def:compiso} that preserves the color of tiles and sends 
  the  $1$-vertices on the top left marked by a large black dot to the $1$-vertices on
  the bottom left marked in the same way. 
  From this one can deduce that  $g(\phi(\tau))=f(\tau)=h_0(f(\tau))$ for each $\tau \in  \DD^1$ (this is closely related to 
  Lemma~\ref{lem:labeluniq}~\ref{item:labeluniq2}). 
  
  This  in turn allows us to define a map $h_1\colon S^2\to S^2$ by
  setting 
  $$h_1(x) \coloneqq \big((g|\phi(\tau))^{-1}\circ h_0 \circ (f|\tau)\big)(x), $$ whenever $x\in S^2$ and $x\in \tau \in \DD^1$. 
 As in the proof of
  Lemma~\ref{lem:isotwotequiv}, one can show that $h_1$ is well-defined. Then $h_1(\tau)=\phi(\tau)$ for each $\tau \in \DD^1$.   
 Lemma~\ref{lem:isocellhomeo}~\ref{item:isocellhomeo1}  implies   that $h_1$ is a homeomorphism on $S^2$.  From the definition of $h_1$ it is also clear that $g\circ h_1=h_0\circ f$.
 
 So in order to conclude that $f$ and $g$ are Thurston
 equivalent, it remains to show that $h_1$ and $h_0=\id_{S^2}$ are isotopic rel.\ $\post(f)$. First note that  the definition of $h_1$ implies that  this map fixes the points in $\post(f)=\post(g)$, i.e., the $0$-vertices.

 Let $\CC'\coloneqq h_1(\CC)\sub S^2$. This is the  Jordan curve 
 drawn on  the bottom left in    Figure~\ref{fig:exp_notcexp}
  with a thick line.  It is intuitively clear and not hard to prove  that $\CC'$ can be deformed into 
  $\CC$ by an isotopy rel.\ $\post(f)$. By using such an isotopy one can show that $h_1$ is isotopic rel.\ $\post(f)$ to a homeomorphism $\widetilde h_0$ on $S^2$ that 
  preserves all $0$-cells as sets (i.e., all cells in $\DD^0(f,\CC)=\DD^0(g,\CC)$).  
  Lemma~\ref{lem:isocellhomeo}~\ref{item:isocellhomeo3}  then implies that  $\widetilde h_0$, and hence also $h_1$, is isotopic to $h_0=\id_{S^2}$ rel.\ $\post(f)$.
  The Thurston equivalence of $f$ and $g$ follows. 
 \end{ex}

\ifthenelse{\boolean{singlechapter}}{

%


\chapter{Invariant  curves}
\label{cha:constructc}

This chapter is central for this work. We will prove existence
and uniqueness results for invariant curves $\CC$ of an expanding
Thurston map $f$.  We will also show that if an invariant curve
exists, then it can be obtained from an iterative procedure, and
that it is a quasicircle.  We always require that $\CC$ is a
Jordan curve and that $\post(f)\sub \CC$, but in the following
discussion we will often refer to such a curve $\CC$ simply as an
invariant curve for brevity.

One of our  main results can be formulated as follows. 

%

\begin{theorem}
  [High iterates have invariant curves]
  \label{thm:main}
  \index{f-invariant@$f$-invariant!Jordan curve}
  \index{Jordan curve!f-invariant@$f$-invariant}
  \index{invariant!Jordan curve}
  Let $f\colon S^2\to S^2$ be an expanding Thurston map, and $\CC\sub S^2$ be a Jordan curve with $\post(f)\sub \CC$.  Then for 
  each sufficiently large $n\in \N$ there exists a Jordan curve $\widetilde{\CC}\sub S^2 $ that is invariant for $f^n$ and isotopic to $\CC$ rel.\  $\post(f)$.
\end{theorem}

This existence result has  the following important implication.

\begin{cor}
  [Thurston maps and subdivision rules]
  \label{cor:subdivnlarge} 
  \index{two-tile subdivision rule!realization of}
  \index{subdivision}
  \index{Thurston map!iterate of}
  \index{iterate of Thurston map}
  \index{F f@$F=f^n$}
  Let $f\: S^2 \ra S^2$ be an expanding Thurston map.
Then for each sufficiently large 
$n\in \N$  there exists a two-tile subdivision rule that is realized by $F=f^n$. 
\end{cor} 

This  justifies  our  approach of studying expanding Thurston maps from a combinatorial perspective based on cellular Markov partitions.

Invariant curves are  quasicircles if  the underlying metric is visual.

 \begin{theorem}[Invariant curves are quasicircles]
  \label{thm:Cquasicircle} 
  \index{f-invariant@$f$-invariant!Jordan curve}
  \index{Jordan curve!f-invariant@$f$-invariant}
  \index{invariant!Jordan curve}
  \index{quasicircle}
  Let $f\: S^2\ra S^2$ be an expanding Thurston map, and  $\CC\sub S^2$ be a Jordan curve with $\post(f)\sub \CC$. If  $\CC$ is $f$-invariant, then $\CC$ equipped with (the restriction of) a visual metric for $f$  is a quasicircle. 
\end{theorem} 

This also applies  to invariant curves of  iterates, because if $f\: S^2\ra S^2$ is an expanding Thurston map, then the same is true for each iterate $F=f^n$, $n\in \N$. 

If one studies rational Thurston maps $f$ that are expanding,
then the underlying $2$-sphere is the Riemann sphere $\CDach$,
and it is natural to equip it with the chordal metric $\sigma$.
Then an invariant $\CC$ curve as in the previous theorem  is also
a quasicircle with respect to $\sigma$.
This  can be deduced from Theorem~\ref{thm:Cquasicircle} once we know that 
for such maps  the chordal metric is quasisymmetrically
equivalent to each visual metric. This will be proved in
Chapter~\ref{cha:geom-visu-sphere}; see in particular
Corollary~\ref{cor:visualqsmetric}.

As we discussed in Section~\ref{sec:Thurtoncurves},  if $\CC$ an $f$-invariant Jordan curve with $\post(f)\sub\CC$, then we get a sequence of cell decompositions $\DD^n=\DD^n(f,\CC)$, $n\in \N_0$, so that each 
 cell decomposition is refined by the  cell decompositions  of higher levels. We will see that 
Theorem~\ref{thm:Cquasicircle} implies that we have   good control for the geometry of edges and tiles in these cell decompositions. Namely,  the family of edges
consists of {\em uniform quasiarcs} and the boundary of tiles are {\em uniform quasicircles}.
See Section~\ref{sec:cc-quasicircle} for an explanation of this terminology and Proposition~\ref{prop:arc} for a precise statement.

In Theorem~\ref{thm:main} it is necessary to pass to an iterate of the map to guarantee existence of an invariant curve, because there are examples of maps for which an invariant curve does not exist (see Example~\ref{ex:noinvCC}). One can formulate a necessary and sufficient 
 criterion for  the existence of invariant curves.

 \begin{theorem}[Existence of invariant curves]
  \label{thm:exinvcurvef}
  \index{f-invariant@$f$-invariant!Jordan curve}
  \index{Jordan curve!f-invariant@$f$-invariant}
  \index{invariant!Jordan curve}
  Let $f\colon S^2\to S^2$ be an expanding Thurston map. 
  Then the following conditions are equivalent:
  
  \begin{enumerate}
  
  \item 
    \label{item:ex_invC1}
    There exists a Jordan curve
     $ \widetilde \CC\sub
    S^2$  with    $\post(f)\sub\widetilde  \CC$ that is $f$-invariant.   
 
  \item 
    \label{item:ex_invC2}
    There exist Jordan curves $\CC,  \CC'\sub S^2$ with 
     $\post(f)\sub \CC, \CC'$ and $ \CC'\sub f^{-1}(\CC)$,  and  an isotopy  $H\colon S^2\times I\to S^2$  rel.\ $\post(f)$ 
   with  $H_0 = \id_{S^2}$ and  $H_1(\CC)= \CC'$ such  that
    the map
    \begin{equation*}
      \widehat{f}\coloneqq  H_1 \circ f \text{ is combinatorially expanding
        for } \CC'.
    \end{equation*}
  \end{enumerate}
Moreover, if \ref{item:ex_invC2} is true, then there exists an $f$-invariant Jordan curve $\widetilde \CC\sub S^2$ with $\post(f)\sub \widetilde \CC$ that is isotopic to $\CC$ rel.\ $\post(f)$ and isotopic to $ \CC'$ rel.\ 
  $f^{-1}(\post(f))$. 
\end{theorem}

The first condition in \ref{item:ex_invC2} says that there exists a Jordan curve $\CC$ 
with $\post(f)\sub \CC$ that can be isotoped rel.\ $\post(f)$ 
into its preimage under $f$. This condition alone ensures that an associated
$f$-invariant set $\widetilde \CC$ with $\post(f)\sub \widetilde \CC$ exists, but in general $\widetilde \CC$ will  not be a Jordan curve (see Lemma~\ref{lem:CCnprop}~\ref{CCprop7} and Example~\ref{ex:Cinv_notcexp}). If, in addition, the map $\widehat f$  is combinatorially expanding as stipulated in \ref{item:ex_invC2}, then one obtains a Jordan curve 
$\widetilde \CC$. 

Invariant curves can be constructed by an iterative procedure
that will be described in Section~\ref{subsec:ittproc}. 
In the situation of Theorem~\ref{thm:exinvcurvef}, one lifts 
the isotopy
$H$ by the map $f$ repeatedly to obtain a sequence of isotopies
$H^n\: S^2\times I\ra S^2$, $n\in \N_0$, with $H^0\coloneqq H$
such that $H^n_0=H^n(\cdot, 0)=\id_{S^2}$.  One sets
$\CC^0\coloneqq \CC$ and defines inductively
$\CC^{n+1}\coloneqq H_1^n(\CC^n) $ for $n\in \N_0$. It can then
be shown that the sequence $\{\CC^n\}$ Hausdorff converges to the
desired invariant curve $\widetilde \CC$ (see
Proposition~\ref{prop:invCit}). An explicit knowledge of the
isotopies is not really necessary, because one can interpret this
as an edge replacement procedure (see
Remark~\ref{rem:C_arc_replace}).  In Section~\ref{subsec:ittproc}
we will discuss this and several examples that illustrate various
phenomena in this context.

Theorem~\ref{thm:main}, which is our basic existence result for
invariant curves, is complemented by the following uniqueness
statement.

\begin{theorem}[Uniqueness of invariant curves] 
  \label{thm:uniqc}     
  \index{f-invariant@$f$-invariant!Jordan curve}
  \index{Jordan curve!f-invariant@$f$-invariant}
  \index{invariant!Jordan curve}
  Let
  $f\: S^2\ra S^2$ be an expanding Thurston map, and $\CC, \CC'\sub S^2$
  be $f$-invariant Jordan curves  that both contain the set
  $\post(f)$.  Then $\CC=\CC'$ if and only if $\CC$ and $\CC'$ are
  isotopic rel.\ $f^{-1}(\post(f))$.
\end{theorem}

This implies that in a given isotopy class rel.\ $\post(f)$ there
are only finitely many invariant curves $\CC$
(Corollary~\ref{cor:finiterelP}). It follows that an expanding
Thurston map $f$ with $\#\post(f)=3$ can have only finitely many
invariant curves $\CC$ (Corollary~\ref{cor:finitepost3}).

The situation changes if one does not restrict the isotopy class
of $\CC$.  An expanding Thurston map $f$ may have 
infinitely many invariant curves $\CC$ in general (see
Example~\ref{ex:infty_C}). If, in addition, the map is rational
and has a hyperbolic orbifold, then this cannot happen and $f$
can have only finitely many invariant curves $\CC$ (see
Theorem~\ref{thm:rat_finitelyC}).

The chapter is organized as follows. Section~\ref{subsec:exuniq} is devoted to existence and uniqueness results, where we provide proofs for the statements discussed above. The iterative procedure for the construction of invariant curves is explained in Section~\ref{subsec:ittproc}. In Section~\ref{sec:cc-quasicircle} we discuss the quasiconformal geometry of invariant curves. Here we prove Theorem~\ref{thm:Cquasicircle} and related results. 

Much of our discussion in this chapter is quite technical. Before we go into the details, we  look at a specific example that will illustrate some of the main ideas.

\ifthenelse{\boolean{nofigures}}{}{
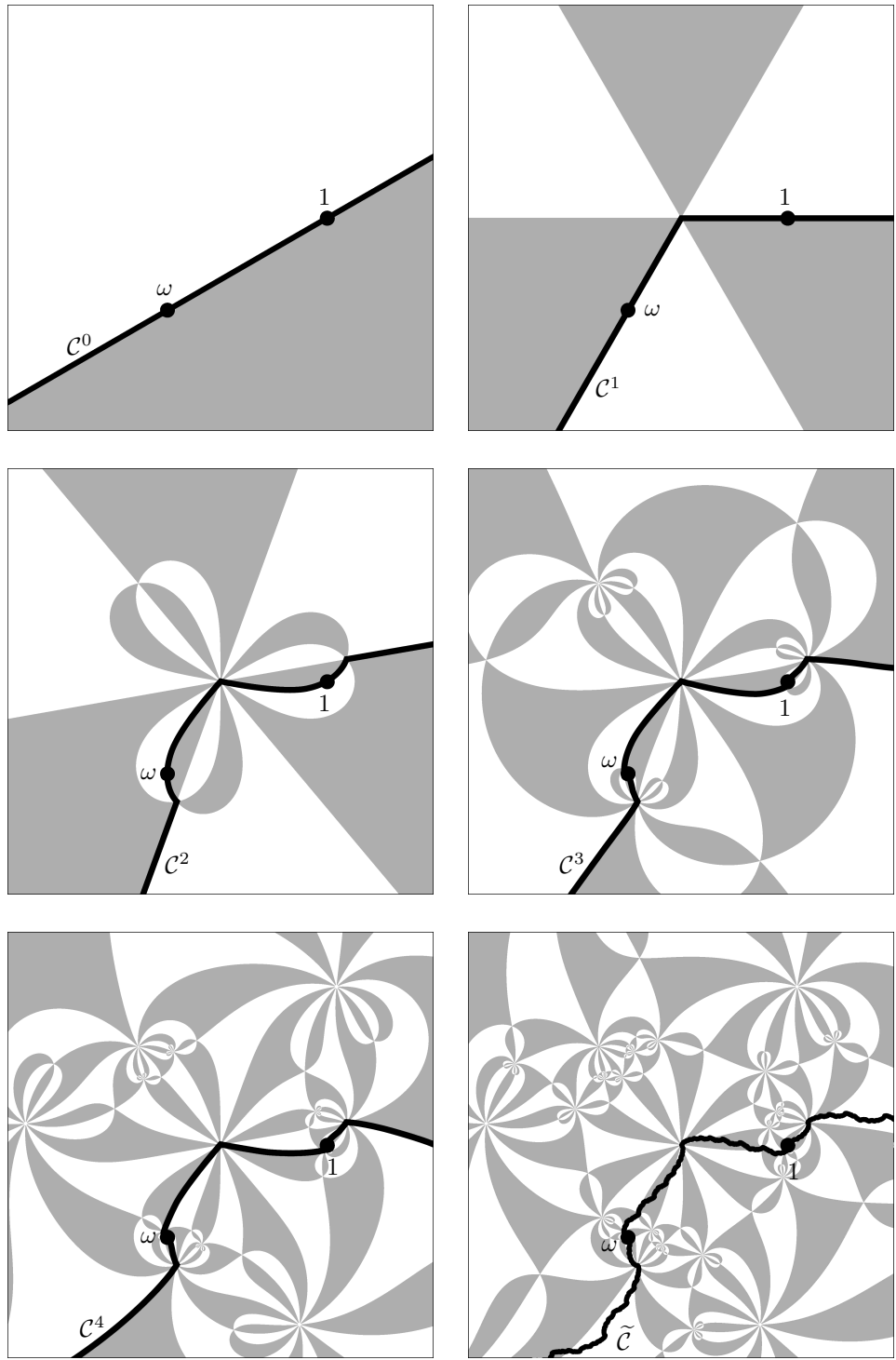
\begin{figure}
  \centering
  \begin{subfigure}{0.48\textwidth}
    \begin{overpic}
      [width=2.395in, 
      tics=20]{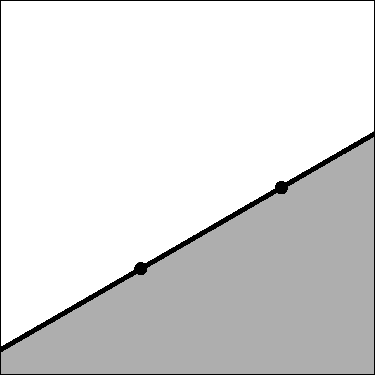}
      \put(35,32){$\omega$}
      \put(73,53){$1$}
      \put(14,18){$\CC^0$}
    \end{overpic}
  \end{subfigure}
  \hspace*{\fill}
  \begin{subfigure}{0.48\textwidth}  
  \begin{overpic}
      [width=2.395in, 
      tics=20]{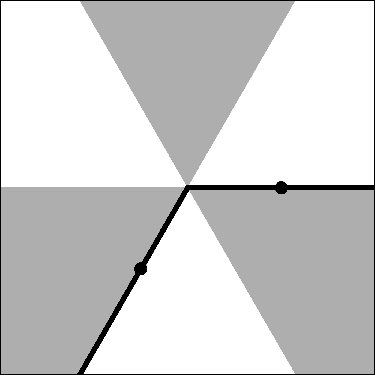}
      \put(41.5,27){$\omega$}
      \put(73,53){$1$}
      \put(30,8){$\CC^1$}
    \end{overpic}
  \end{subfigure}
  
  \vspace{0.04\textwidth}
  \begin{subfigure}{0.48\textwidth}
    \begin{overpic}
      [width=2.395in, 
      tics=20]{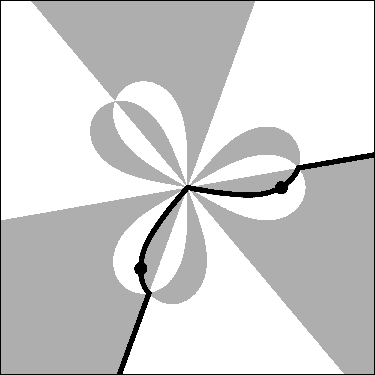}
      \put(31.3,27){$\omega$}
      \put(73,43){$1$}
      \put(37,5){$\CC^2$}
    \end{overpic}
  \end{subfigure}
  \hspace*{\fill}
  \begin{subfigure}{0.48\textwidth}
    \begin{overpic}
      [width=2.395in, tics=20]{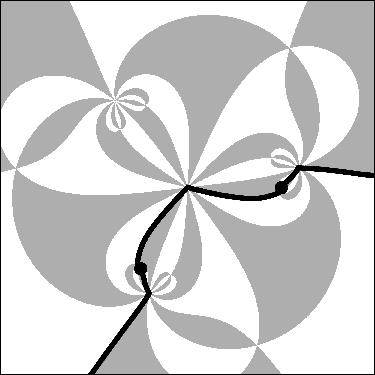}
      \put(31.3,30){$\omega$}
      \put(73,42){$1$}
      \put(21.5,5){$\CC^3$}
    \end{overpic}
  \end{subfigure}
  
  \vspace{0.04\textwidth}
  \begin{subfigure}{0.48\textwidth}
    \begin{overpic}
      [width=2.395in, tics=20]{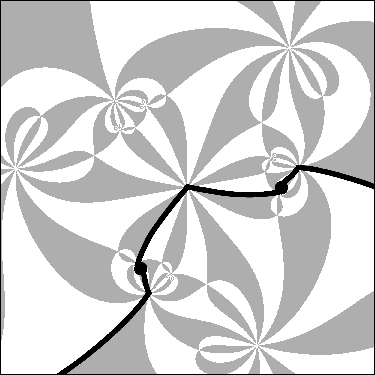}
      \put(31.3,27){$\omega$}
      \put(75,43){$1$}
      \put(17,5){$\CC^4$}
    \end{overpic}
    \end{subfigure}
    \hspace*{\fill}
  \begin{subfigure}{0.48\textwidth}
    \begin{overpic}
      [width=2.395in, tics=20]{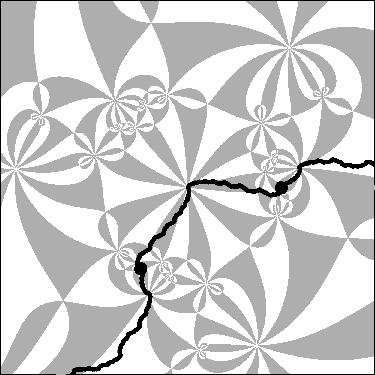}
      \put(31.3,25){$\omega$}
      \put(75,42){$1$}   
      \put(35,2){$\widetilde{\CC}$}
    \end{overpic}
  \end{subfigure}
  \caption{The invariant curve for Example \ref{ex:invC}.}
  \label{fig:invC_constr} 
\end{figure}


}

\begin{ex}
  \label{ex:invC}
  Let  $f\: \CDach \ra \CDach$ be the map defined by 
  $$f(z)= 1+ (\omega-1)/z^3$$ for $z\in \CDach$, where $\omega= e^{4\pi
    \iu/3}$. This map was already considered in
  Example~\ref{ex:tringflP} and Example~\ref{ex:R_mario3}. It realizes
  the two-tile subdivision rule shown in
  Figure~\ref{fig:triangle_flap0}. 

  Note that $f(z)= \tau(z^3)$, where
  $\tau(w)= 1+ (\omega-1)/w$ is a M\"{o}bius transformation that maps
  the upper half-plane to the half-plane above the line through the points 
  $\omega$ and $1$ (indeed,  $\tau$ maps $0,1,\infty$ to $\infty,
  \omega,1$, respectively).  
We have $\crit(f)=\{0,\infty\}$ and  $\post(f)=\{\omega, 1, \infty\}$. 
 
One can obtain an $f$-invariant Jordan curve $\widetilde \CC\sub \CDach$ with $\post(f)\sub \widetilde \CC$ as follows.  We first pick a  Jordan
  curve $\CC^0\sub \CDach$ containing all postcritical points of $f$. More specifically, 
  let  
  $\CC^0$  be the (extended) line through $\omega$ and $1$ (i.e., the circle on $\CDach$ 
  through $\omega,1, \infty$). 

Now   consider  $f^{-1}(\CC^0)= \bigcup_{k=0,\dots, 5} R_k$, where
  $$R_k=\{re^{\iu k\pi/3} : 0\leq r\leq \infty\}$$ is the   ray  from $0$ through 
  the sixth root of unity $e^{\iu k\pi /3}$; see the top right in Figure
  \ref{fig:invC_constr}. We choose a Jordan curve $\CC^1\sub \CDach$ such that
 $$    \CC^1\subset f^{-1}(\CC^0),\ 
     \post(f)\subset \CC^1, \text{ and }
   \CC^1 \text{ is isotopic to $\CC^0$ rel.\ $\post(f)$.} 
 $$
  For general Thurston maps a similar choice is not always possible, but in our specific case 
   there is a unique Jordan  curve $\CC^1\subset f^{-1}(\CC^0)$ with  
     $\post(f)\subset \CC^1$, namely $\CC^1=R_0\cup R_4$, the
  union of the two rays through $\omega$ and through  $1$. 
Since $\#\post(f)= 3$,  the requirement that $\CC^1$ is isotopic  to $\CC^0$ rel.\ $\post(f)$ is automatic for   our specific map  $f$ by
  Lemma \ref{lem:deform<4}. 
  Let $H\colon \CDach\times I\to \CDach$ be an  isotopy rel.\ $\post(f)$
  that deforms $\CC^0$ to $\CC^1$, i.e., $H_0=\id_{\CDach}$  and  $H_1(\CC^0)=
  \CC^1$. 

  Given the data $\CC^0$, $\CC^1$, and $H$, there are two (essentially equivalent)   ways to obtain  an 
  $f$-invariant Jordan curve isotopic  to $\CC^1$ rel.\
  $f^{-1}(\post(f))$ and hence also isotopic to $\CC^0$ rel.\ $\post(f)$.

 For the first approach we consider the Thurston map
  $\widehat{f}\coloneqq  H_1\circ f$. Since $\CC^1\sub f^{-1}(\CC^0)$ we have $f(\CC^1)\sub \CC^0$, and so  
  $$\widehat{f}(\CC^1)= (H_1\circ
  f)(\CC^1) \subset H_1(\CC^0)=\CC^1. $$ Thus $\CC^1$ is
  $\widehat f$-invariant. The two-tile subdivision rule given by
  $\DD^1=\DD^1(\widehat{f}, \CC^1)$,
  $\DD^0=\DD^0(\widehat{f},\CC^1)$, and the labeling induced by
  $\widehat f$ is as in Figure~\ref{fig:triangle_flap0}. 
The map  $\widehat{f}$ is combinatorially
expanding for $\CC^1$; indeed,  no $2$-tile for $(\widehat f, \CC^1)$ joins opposite sides of $\CC^1$.  Thus by
  Theorem~\ref{thm:combexp1} there is a homeomorphism $\phi\colon
  \CDach\to \CDach$ isotopic to the identity on 
  $\CDach$ rel.\ $\post(\widehat f)=
  \post( f)$ such that
  $\phi(\CC^1)= \CC^1$ 
and $g\coloneqq \phi\circ \widehat{f}$ is an expanding Thurston map.
Since $f$ is also expanding (as follows from Proposition~\ref{prop:rationalexpch}) and  $g$ is Thurston equivalent to $f$, there is
  a homeomorphism $h\colon \CDach\to \CDach$ such that $h\circ f= g\circ
  h$ (Theorem~\ref{thm:exppromequiv}). Then $\widetilde{\CC}\coloneqq  h^{-1}(\CC^1)$ is an $f$-invariant
  Jordan curve containing $\post(f)$. The general existence result for invariant curves given by  Theorem \ref{thm:exinvcurvef}
  is proved in the same
  way.

For the second approach, we use  
Proposition \ref{prop:isotoplift} to lift 
  $H=H^0$ by the map $f$ to an isotopy  $H^1$ with $H^1_0=\id_{\CDach}$. Then we lift $H^1$ to an isotopy $H^2$ with $H^2_0=\id_{\CDach}$, etc. In this way,   we  find a sequence of isotopies $H^n$ and 
  inductively define $\CC^{n+1}\coloneqq  H_1^n(\CC^n)$. We will see in Proposition
  \ref{prop:invCit} that the sequence $\{\CC^n\}$ of
  Jordan curves Hausdorff  converges  to an $f$-invariant Jordan
  curve $\widetilde{\CC}$ containing all postcritical points of $f$ as
  desired. This is illustrated in Figure \ref{fig:invC_constr};
  indeed, the invariant curve $\widetilde \CC$ in this  picture was obtained by approximating it by the curves $\CC^n$  (as were the invariant curves in 
  Figures~\ref{fig:Cit},~\ref{fig:Cinv_notexp}, and~\ref{fig:not_rect_invC}).

In our example the $f$-invariant Jordan
  curve $\widetilde \CC\sub \CDach$ with $\post(f)\subset\widetilde{\CC}$ is in fact {\em unique}. 
 To see this, note that  since $\#\post(f)=3$, every such curve $\widetilde \CC$ is isotopic  rel.\ $\post(f)$ to the curve $\CC^0$ chosen above. Thus we can find an isotopy $K\: \CDach\times I\ra \CDach$ rel.\ $\post(f)$ with $K_0=\id_{\CDach}$ and $K_1(\widetilde \CC)=\CC^0$.  By Proposition \ref{prop:isotoplift} we can lift $K$ to an  isotopy $\widetilde K\: \CDach\times I\ra \CDach$ rel.\ $f^{-1}(\post(f))$ with $\widetilde K_0=\id_{\CDach}$ and 
  $K_t\circ f=f\circ \widetilde K_t$ for $t\in I$. Then by Lemma~\ref{lem:lifts_inverses} we have 
  $$\CC'\coloneqq  \widetilde K_1(\widetilde \CC)\sub \widetilde K_1(f^{-1}(\widetilde \CC))=f^{-1}(K_1(\widetilde \CC))=f^{-1}(\CC^0).$$
  So $\CC'$ is a Jordan curve in $\CDach$ with $\CC'\sub
  f^{-1}(\CC^0)$ and $\post(f)\sub \CC'$. Since in this
  particular example $\CC^1$ is the unique such curve, we conclude $\CC'=\widetilde K_1(\widetilde\CC)=\CC^1$. In particular, $\widetilde \CC$ is isotopic to $\CC^1$ rel.\ $f^{-1}(\post(f))$ by the isotopy $\widetilde K$. 
  So every $f$-invariant Jordan curve $\widetilde \CC$ with $\post(f)\sub \widetilde \CC$ lies in the same isotopy class rel.\ $f^{-1}(\post(f))$ as 
  $\CC^1$. Hence by  
   Theorem \ref{thm:uniqc} (which we will prove momentarily) there is  at most one such Jordan curve  $\widetilde \CC$. The uniqueness of $\widetilde \CC$ follows.  
\end{ex}

In Example~\ref{ex:Cit} the reader can find another illustration
for the construction of an invariant curve (see
Figure~\ref{fig:Cit}).

\section{Existence and uniqueness of invariant curves} 
\label{subsec:exuniq} 
We now turn to general expanding Thurston maps and establish
existence and uniqueness results for invariant curves. We start
with uniqueness results.

\begin{proof} [Proof of Theorem~\ref{thm:uniqc}.] Suppose  $f\:S^2\ra S^2$ is  an expanding Thur\-ston map, and    $\CC$
and $\CC'$  are   $f$-invariant Jordan curves in $S^2$ that both contain the set  $\post(f)$ and  are isotopic rel.\ $f^{-1}(\post(f))$. We have to show that $\CC=\CC'$.

 Under the given assumptions,  there exists an isotopy $H^0\: S^2\times 
I\ra S^2$ rel.\ $f^{-1}(\post(f))$ with $H^0_0=\id_{S^2}$ and $H^0_1(\CC)=\CC'$. Since $\post(f)\sub f^{-1}(\post(f))$, the map $H^0$ is also an isotopy  rel.\ $\post(f)$. 
Hence by Proposition~\ref{prop:isotoplift} we can find an isotopy 
$H^1\: S^2\times 
I\ra S^2$ rel.\ $f^{-1}(\post(f))$ with $H^1_0=\id_{S^2}$ and $f\circ H^1_t=H^0_t\circ f$ for  $t\in I$.  Repeating this argument, we obtain isotopies  $H^n\: S^2\times 
I\ra S^2$ rel.\ $f^{-1}(\post(f))$ with $H^n_0=\id_{S^2}$ and $f\circ H^{n+1}_t=H^n_t\circ f$ for $t\in I$ and $n\in \N_0$.  

\smallskip
{\em Claim.} $H^n_1(\CC)=\CC'$ for  $n\in \N_0$. 
\smallskip

To see this, we use induction on $n$. 
For $n=0$ the claim is true by choice of $H^0$. 

Suppose that $H^n_1(\CC)=\CC'$ for some $n\in \N_0$. 
Then Lemma~\ref{lem:lifts_inverses}
and the identity  $f\circ H^{n+1}_1=H^n_1\circ f$ imply  that  
$$H^{n+1}_1(f^{-1}(\CC))=f^{-1}(H^n_1(\CC))=f^{-1}(\CC').$$ Since $\CC$ and $\CC'$ are   $f$-invariant, we have the inclusions  
$\CC\sub f^{-1}(\CC)$  and $\CC'\sub f^{-1}(\CC')$. In particular, 
$$\widetilde \CC\coloneqq  H^{n+1}_1(\CC)\sub
H^{n+1}_1(f^{-1}(\CC))=f^{-1}(\CC')$$ is a Jordan curve contained
in $f^{-1}(\CC')$. Moreover, 
the curves
$\CC$ and 
$\widetilde \CC $ are isotopic rel.\  $f^{-1}(\post(f))$ (by the isotopy $H^{n+1}$). Since $\CC$ and $\CC'$ are isotopic rel.\ $f^{-1}(\post(f))$ by our hypotheses,   it follows that $\CC'$ and $\widetilde \CC$ are also isotopic rel.\ $f^{-1}(\post(f))$.  Both sets are contained in  $f^{-1}(\CC')$. 

Now $f^{-1}(\CC')$ is the $1$-skeleton of the cell decomposition $\DD^1(f, \CC')$. This cell decomposition has the vertex set $f^{-1}(\post(f))$. Moreover, since $f$ is expanding,  $\#\post(f)\ge 3$, and so every tile in $\DD^1(f,\CC')$ has at least three vertices. So the hypotheses of
Lemma~\ref{lem:isoJcin1ske} are satisfied and we conclude that
$\CC'=\widetilde \CC=H^{n+1}_1(\CC)$. The claim above follows. 

\smallskip
 Now fix a visual metric on $S^2$. Then  the tracks  of the isotopies $H^n$ shrink at an exponential rate  as $n\to \infty$ (Lemma~\ref{lem:exp_shrink}). Since $H^n_0=\id_{S^2}$, it follows that $H^n_1\to \id_{S^2}$ uniformly as $n\to \infty$. 
 Since $H^n_1(\CC)=\CC'$ for all $n\in \N_0$ by the claim, we conclude
 $\CC=\CC'$ as desired.
\end{proof}

\begin{cor}[Invariant curves rel.\ $\post(f)$]  
\label{cor:finiterelP}
  \index{f-invariant@$f$-invariant!Jordan curve}
  \index{Jordan curve!f-invariant@$f$-invariant}
  \index{invariant!Jordan curve}
Let $f\: S^2\ra S^2$ be an expanding Thurston map, and 
$\CC\sub S^2$ be a Jordan  curve with $\post(f)\sub \CC$. Then there are at most finitely many $f$-invariant Jordan curves $\widetilde \CC\sub S^2 $ with $\post(f)\sub \widetilde \CC$  that are isotopic to $\CC$ rel.\ $\post(f)$. 
\end{cor}

\begin{proof}   Let $\widetilde \CC$ be such an $f$-invariant Jordan curve. 
Then there exists an isotopy $H\:S^2\times I\ra S^2$ rel.\ $\post(f)$ with $H_0=\id_{S^2}$ and $H_1(\widetilde \CC)=\CC$. Lifting $H$ we get an isotopy 
$\widetilde H\: S^2\times I\ra S^2$ rel.\ $f^{-1}(\post(f))$ such that 
$\widetilde H_0=\id_{S^2}$ and $f\circ \widetilde H_t=H_t\circ f$ for  $t\in I$.  Since $\widetilde \CC$ is $f$-invariant, we have $\widetilde \CC\sub f^{-1}(\widetilde \CC)$. So Lemma~\ref{lem:lifts_inverses} implies that 
$$\widetilde H_1(\widetilde \CC)\sub \widetilde H_1(f^{-1}(\widetilde \CC))= f^{-1}(H_1(\widetilde \CC))=f^{-1}(\CC).$$  Hence    $\widetilde \CC$ is isotopic rel.\ $f^{-1}(\post(f))$ to the  Jordan curve 
$\widetilde H_1(\widetilde \CC)$ that is contained in $f^{-1}(\CC)$. 
Any such Jordan curve is a union of edges in the cell decomposition $\DD^1(f, \CC)$ (see the last part of the proof of Lemma~\ref{lem:isoJcin1ske}).  In particular, there are only finitely many distinct Jordan curves contained in $f^{-1}(\CC)$. This implies that there are only finitely many isotopy classes rel.\ $f^{-1}(\post(f))$ represented by curves $\widetilde \CC$ satisfying the assumptions of the corollary. Since an $f$-invariant Jordan curve $\widetilde \CC\sub S^2$ with $\post(f)\sub\widetilde  \CC$ is unique in its isotopy class rel.\ $f^{-1}(\post(f))$ by Theorem~\ref{thm:uniqc}, the  statement follows. 
\end{proof}

\begin{cor}\label{cor:finitepost3}
 Suppose $f\: S^2\ra S^2$ is an expanding Thurston map with 
 $\#\post(f)=3$. 
Then there are at most finitely many $f$-invariant Jordan curves   
$\widetilde \CC\sub S^2$ with  $\post(f)\sub \widetilde \CC$. 
\end{cor}

\begin{proof} Pick a  Jordan curve $\CC\sub S^2$ with $\post(f)\sub \CC$.
Since we have $\#\post(f)=3$, by Lemma~\ref{lem:deform<4} every   Jordan curve $\widetilde \CC\sub S^2$ with $\post(f)\sub \widetilde\CC$ is isotopic to $\CC$ rel.\ $\post(f)$. 
The statement now  follows from Corollary~\ref{cor:finiterelP}.
\end{proof}

In contrast to the case $\#\post(f)=3$,  expanding Thurston maps $f$ with $\#\post(f)\ge 4$ can have infinitely many distinct
invariant  curves.

\ifthenelse{\boolean{nofigures}}{}{
\begin{figure}
  \centering
  \begin{overpic}
    [width=12cm, 
    tics=20]{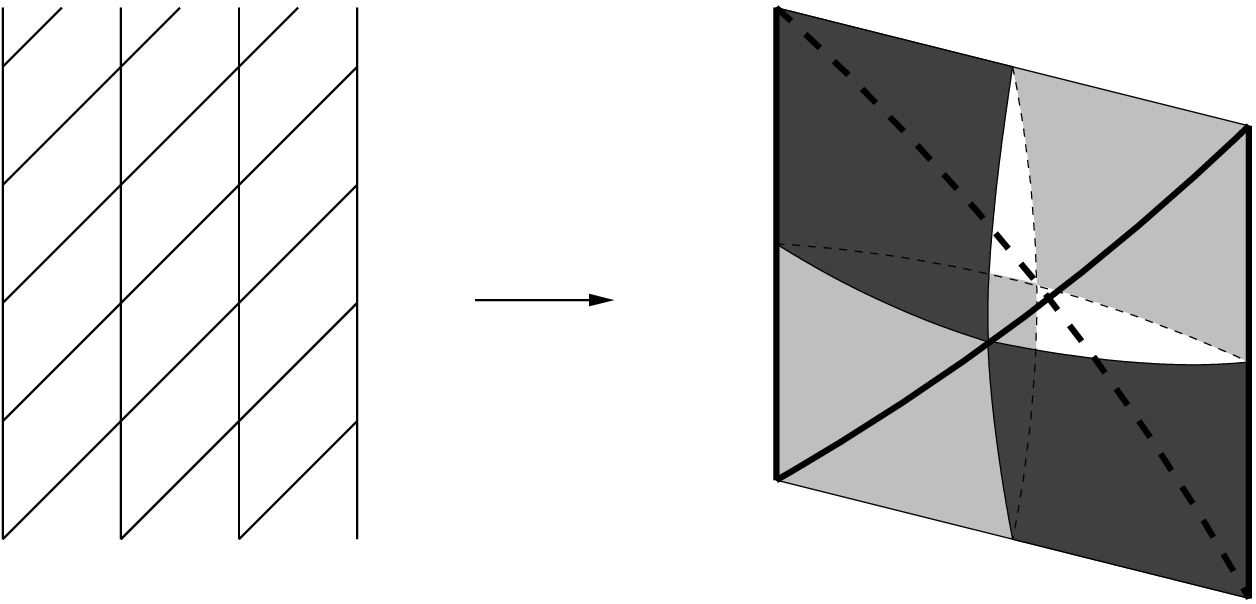}
    \put(42,26){$\Theta$}
    \put(32,42){$\R^2$}
    \put(13,5){$K_M$}
    \put(57,15){$\CC_M$}
  \end{overpic}
  \caption{Invariant curves for the Latt\`{e}s map $g$.}
  \label{fig:inf_invC}
\end{figure}
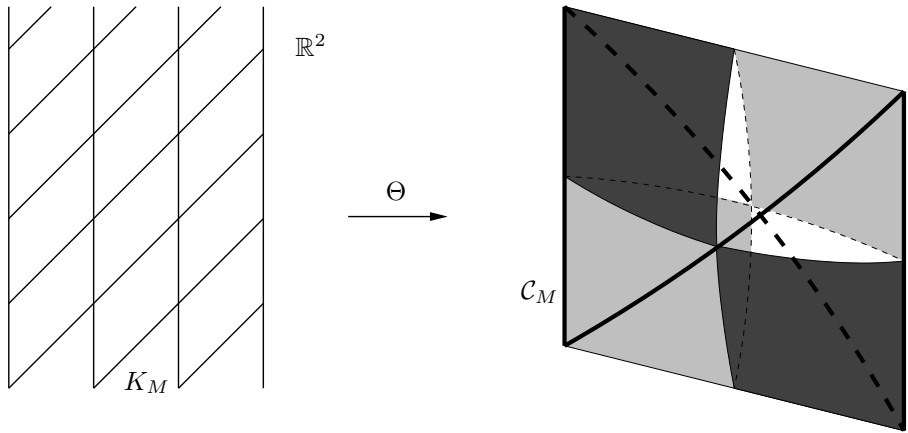
}

\begin{ex}[Infinitely many invariant curves]
  \label{ex:infty_C}
  Let $f$ be the  Latt\`{e}s map  from Section
  \ref{sec:Lattes} (there called $g$). In the following, it is advantageous to use real notation as in Example~\ref{ex:lattes_type} and consider the maps $A$ and $\Theta$ used in the definition of $f$ as in  \eqref{eq:Lattes} as maps on $\R^2$. Then 
    $A(u)= 2u$ for $u\in \R^2$. Let $G$ be the crystallographic group consisting of all maps 
    $u\in \R^2\mapsto g(u)=\pm u+\ga$, where  $\ga \in \Gamma\coloneqq \Z^2$. Then $\Theta$ is induced by $G$ and so   for $u_1,u_2\in \R^2$ we have
    $\Theta(u_1)=\Theta(u_2)$ if and only if  there exists $g\in G$ with $u_2=g(u_1)$. 
    
    Let $S=[0,1/2]^2\sub \R^2$.  
   Recall that the extended real line $\widehat{\R}=\Theta(\partial S)$ (which is the equator
  of the pillow) is $f$-invariant and 
  contains $\{-1,0,1,\infty\}=
  \post(f)=\Theta(\frac12 \Gamma)$. 
  
 Consider the square grid $K$ given as the union of the horizontal and vertical lines in $\R^2$ 
  that pass through a  point in $\frac{1}{2}\Gamma=\frac 12\Z^2$. Note that $K=\bigcup_{g\in G} g(\partial S)$. 
   The map  
   $\Theta|\partial S$ is injective and $\Theta(\partial S)=\Theta(K)=\widehat \R$. So the $f$-invariant curve $\widehat \R$ is   the image of   $K$  under $\Theta$. 
   One can obtain other  $f$-invariant Jordan curves by mapping 
    other grids by $\Theta$. 
   
  To explain  this, we consider a $(2\times 2)$-matrix  $M\in \text{SL}_2(\Z)$  
  (here $\text{SL}_2(\Z)$ 
denotes the set 
of
  $(2\times2)$-matrices with integer entries and determinant
  $1$).  We identify $M$ with the linear map $u\mapsto M u$ on
  $\R^2$ induced by left-multiplication of $u\in \R^2$
  (considered as a column vector) by the matrix $M$. Then $M\circ
  G\circ  M^{-1}=G$, i.e., 
$G$ is invariant under conjugation by $M$.
 
  Now let  $S_M\coloneqq M(S)$, and define the  corresponding
  grid $K_M\coloneqq M(K)=\bigcup_{g\in G}g(\partial S_M)$. 
Since conjugation by $M$ preserves $G$,
we see that  
$\Theta(M(u_1))=\Theta(M(u_2))$
  for $u_1,u_2\in \R^2$  if and only if there exists $g\in G$ with $u_2=g(u_1)$. This implies that $\Theta|\partial S_M$ is injective, and so $\CC_M\coloneqq \Theta(\partial S_M)\sub \CDach$ is a Jordan curve. 
 Moreover, $\Theta(K_M)=\Theta(\partial S_M)$. Since $A\circ M=M\circ A$, 
 we have 
 $$A(K_M)=A(M(K))=M(A(K))\sub M(K)=K_M,$$
 and so 
 $$ f(\CC_M)=f(\Theta(K_M))=\Theta(A(K_M))\sub \Theta(K_M)=\CC_M. $$
 Hence  $\CC_M$ is $f$-invariant. Since $\frac12 \Gamma=M(\frac 12\Gamma)\sub K_M$, we also have $\post(f)=\Theta(\frac12 \Gamma)\sub \Theta(K_M)=\CC_M$. 
 So $\CC_M$ is an $f$-invariant Jordan curve that contains the set 
 $\post(f)$.  An example of  this  construction is indicated in Figure \ref{fig:inf_invC}. The curve   $\CC_M$ is drawn 
 in thick  on the right.    
 
 The curve $\CC_M$ determines the grid $K_M$ uniquely; indeed, 
 one obtains generating vectors of the two lines in  $K_M$ through 
 $0$ by locally lifting  $\CC_M$ near  $\Theta(0)=0\in \post(f)\sub \CC_M$ 
 to $0$ by the map $\Theta$. The whole grid $K_M$ is obtained by translating these two lines by vectors in $\frac12 \Gamma$. 
 
This implies that the map $M \in \text{SL}_2(\Z)\mapsto \CC_M$ is four-to-one; indeed, if $M, N\in \text{SL}_2(\Z)$, then, as we have seen, $\CC_M=\CC_N$ if and only if  $K_M=K_N$. On the other hand,
$K_M=K_N$ if and only if  $M^{-1}\circ N\in \text{SL}_2(\Z)$ is one of the four rotations
around $0$ (by integer multiples of $\pi/2$) that preserve the grid $K$. In particular, there exist infinitely many $f$-invariant Jordan curves 
 $\widetilde \CC\sub \CDach$ with $\post(f)\sub \widetilde
 \CC$.  
 
By 
Proposition~\ref{prop:2222} the map $M$ 
descends to an orientation-preserving 
homeomorphism $h\: \Cdach\ra \Cdach$ such that $h\circ \Theta=\Theta\circ M$. It easily follows
from the above considerations  that $h$ is an {\em automorphism}
of $f$ in the sense that $f\circ h=h\circ f$. So our map $f$ (as each flexible Latt\`es map) has a large associated group formed by  these automorphisms. 
 This is the deeper underlying reason why infinitely many $f$-invariant  Jordan curves exist.     
\end{ex} 

The map $f$ in the previous example is very
special, since it is a flexible Latt\`{e}s map.
In contrast, we have the following
result (pointed out to us by K.~Pilgrim).  

\begin{theorem}
  \label{thm:rat_finitelyC}
  \index{f-invariant@$f$-invariant!Jordan curve}
  \index{Jordan curve!f-invariant@$f$-invariant}
  \index{invariant!Jordan curve}
  \index{Thurston map!rational}
  Let $f\colon \CDach\to \CDach$ be a rational  Thurston map. Suppose that $f$ is expanding and has a hyperbolic orbifold.  Then there are at most finitely many $f$-invariant Jordan curves
  $\CC\subset \CDach$ with $\post(f)\subset \CC$.  
\end{theorem}

\begin{proof} 
  If $\CC\subset \CDach$ is an $f$-invariant Jordan curve with
  $\post(f)\subset \CC$, then we have an associated two-tile
  subdivision rule $(\DD^1, \DD^0, L)$ according to
  Proposition~\ref{prop:ThmapSub}. 
  Recall that this means that
  $\DD^1=\DD^1(f,\CC)$, $\DD^0=\DD^0(f, \CC)$, 
and $L\: \DD^1\ra \DD^0$ is the labeling induced by $f$ (i.e.,
  $L(\tau)=f(\tau)\in \DD^0$ for $\tau \in \DD^1$). 
  
  The number of
  cells in $\DD^1$  is bounded by a constant only
  depending on $\deg(f)$ and $\#\post(f)$. In particular, we have
  a uniform bound independent of $\CC$. This implies that among
  the two-tile subdivision rules obtained in such a way from
  $f$-invariant curves $\CC$, there are only finitely many up to
  isomorphism (see  the discussion before
  Lemma~\ref{lem:isotwotequiv} and  
  Remark~\ref{rem:isomsub} (iii)). 
  Here we use the strong notion
  of 
  isomorphism where we require that the cell complex isomorphisms
  as in the definition of an isomorphism between two-tile
  subdivision rules send positively-oriented flags to
  positively-oriented flags (see Remark~\ref{rem:isomsub}~(ii)).

 
  This allows us to  pick a finite  family $\mathcal{F}$    of 
such curves $ \CC$ 
such that the associated two-tile subdivision rule of any $f$-invariant Jordan curve 
$ \widetilde\CC\subset \CDach$  with $\post(f)\subset \widetilde \CC$ is isomorphic to one
 associated with a curve in  $\mathcal{F}$.
 
 Let $\mathcal{G}$ be the family of all M\"obius transformations $\varphi\: \CDach \ra \CDach$ 
 with $\varphi\circ f =f\circ \varphi$. If $\varphi\in \mathcal{G}$, then 
 $\varphi(\post(f))=\post(f)$. Now $f$ is expanding and so
 $\#\post(f)\ge 3$. 
Since M\"obius transformations are 
uniquely determined by images of three distinct points in $\CDach$,  this implies  that 
 each $\varphi\in \mathcal{G}$ is uniquely determined by the bijection it induces on $\post(f)$. Since there are only finitely many such bijections, $\mathcal{G}$ consists of finitely many elements.

 Now let $\widetilde \CC\sub \CDach$ be an arbitrary $f$-invariant Jordan curve with $\post(f)\sub \widetilde \CC$. We claim that $\widetilde \CC$ is isotopic rel.\ $\post(f)$ to one of the finitely many Jordan curves $\varphi( \CC)$, where $ \CC \in \mathcal{F}$ and $\varphi\in \mathcal{G}$. Since each isotopy class rel.\ $\post(f)$ contains only finitely many $f$-invariant Jordan curves $\CC\sub \CDach$ with 
 $\post(f)\sub \CC$ (Corollary~\ref{cor:finiterelP}), this claim implies  the theorem. 
 
To prove  the claim, we use the fact that the two-tile subdivision rule associated with 
$\widetilde \CC$ is isomorphic to one of the two-tile subdivision rules associated with a curve $ \CC \in \mathcal{F}$. So by Lemma~\ref{lem:isotwotequiv} and 
Remark~\ref{rem:isomsub} there exist orientation-preserving homeomorphisms $h_0,h_1\: \CDach \ra \CDach$ that are isotopic rel.\ $\post(f)$ such $h_0\circ f=f \circ h_1$ and $h_0( \CC)=h_1( 
\CC)=\widetilde \CC$. By Thurston's uniqueness theorem (Theorem~\ref{thm:Thurstonuniqness}) there exists a M\"obius transformation 
$\varphi\:\CDach \ra \CDach$  that is isotopic to $h_0$ rel.\ $\post(f)$ such that $\varphi\circ f=f \circ \varphi$. Then $\varphi\in \mathcal{G}$,  and $\varphi(\CC)$ is isotopic to $h_0(\CC)=\widetilde \CC$
rel.\ $\post(f)$  as desired.  
\end{proof}

Using similar arguments as in the previous proof together with
Theorem~\ref{thm:exppromequiv}, one can   show that
for  an expanding Thurston map $f\colon S^2\to S^2$ with
infinitely many invariant curves there are infinitely many
homeomorphisms $h\colon S^2\to S^2$ with $h\circ f = f\circ h$. 

We now turn our attention to existence results. As the following example shows, for
an expanding   Thurston map $f\:S^2\ra S^2$  an $f$-invariant Jordan curve
$\CC\sub S^2$ with $\post(f)\sub \CC$ need not exist.  

\begin{ex}\label{ex:noinvCC}
Consider  the map $f\: \CDach \ra \CDach$ defined by 
$$f(z)=\iu\frac{z^4-\iu}{z^4+\iu}$$ for $z\in \CDach$.
The
critical points  of $f$ are $0$ and $\infty$, and the map has the following 
ramification portrait:
\begin{equation}
  \label{eq:ramification_h}
  \xymatrix @R=1pt{
    0 \ar[r]^{4 : 1} & -\iu \ar[dr]  &
    \\
    &  & 1\rlap{.}  \ar@(r,u)[]
    \\
    \infty \ar[r]^{4 : 1} & \iu \ar[ur] & 
  }
\end{equation}
So 
the set of postcritical 
points  of $f$ is given by $\post(f)=\{-\iu,1,\iu \}$, and $f$ is a Thurston
map. This map is also expanding as follows from
Proposition~\ref{prop:rationalexpch}.  
\end{ex}

\begin{lemma}
  \label{lem:g1noinvC} Let $f$ be the map from Example \ref{ex:noinvCC}.
  Then there is no $f$-invariant Jordan curve $\widetilde{\CC}\sub
  \CDach $ with $\post(f)\sub \widetilde{\CC}$. 
 \end{lemma}

\ifthenelse{\boolean{nofigures}}{}{
\begin{figure}
  \centering
  \begin{overpic}
    [width=7cm, 
    tics=20]{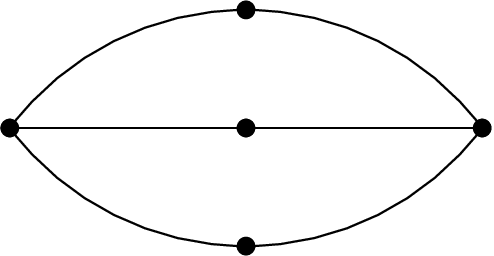}
    \put(52,28){$1$}
    \put(-1,29){$0$}
    \put(99,29){$\infty$}
    \put(46,5){$-\iu$}
    \put(50,43){$\iu$}
    \put(22,19){$R_0$}
    \put(12,43){$R_2$}
    \put(12,7){$R_6$}
  \end{overpic}
  \caption{No invariant Jordan curve $\widetilde{\CC}\supset \post(f)$.}
  \label{fig:no_invC}
\end{figure}
}

\begin{proof}  
We have  $f(z)= \varphi(z^4)$ for $z\in \CDach$, where 
\begin{equation} \label{eq:mobius1}
\varphi(w)= \iu \frac{w-\iu}{w+\iu}, \quad w\in \CDach,
\end{equation}
 is
  a M\"{o}bius transformation that maps the upper half-plane to the
  unit disk (note that $\varphi$ maps $0,1,\infty$ to $-\iu, 1,
  \iu$, respectively). Let $\CC\coloneqq  \partial \D$ be the unit circle. Then
  \begin{equation*}
    f^{-1}(\CC)=\bigcup_{k=0, \dots, 7} R_k, 
    \text{ where } 
    R_k= \{re^{\iu k\pi/4} : 0\leq r\leq \infty\}.   
  \end{equation*}
  The postcritical points
  $-\iu,1,\iu$ lie on distinct rays $R_k$. Two such rays  have the  points $0$ and $\infty$ in common and no other points. Thus
  there is no Jordan curve in $f^{-1}(\CC)$ containing all postcritical
  points (see Figure \ref{fig:no_invC}).  It follows from the considerations in  
 Remark~\ref{rem:iterate_given_Ct} that   the existence of such a Jordan curve is a necessary condition for the existence of an $f$-invariant Jordan curve $\widetilde \CC\sub \CDach$ with $\post(f)\sub \widetilde \CC$ (in our specific  case, 
  where  
    $\#\post(f)=3$, the choice of $\CC$ does not matter, since all Jordan curves that contain $\post(f)$ are isotopic rel.\ $\post(f)$). Hence  there is  no $f$-invariant Jordan curve $\widetilde \CC\sub \CDach$ with $\post(f)\sub \widetilde \CC$. 
  One can also see this by  a simple argument directly. 
  
Indeed, 
  suppose that $\widetilde{\CC}\sub\CDach $ is a Jordan
  curve with $\post(f)\sub \widetilde{\CC}$ and
  $f(\widetilde{\CC})\sub \widetilde{\CC}$.  The unit circle 
  $\CC=\partial \D$  is also a Jordan curve 
  containing the set $\post(f)$. Hence by Lemma~\ref{lem:deform<4} there exists an isotopy $H\: \CDach \times I\ra
  \CDach$ rel.\ $\post(f)$ such that $H_0=\id_{\CDach} $ and
  $H_1(\widetilde{\CC})=\CC$.   
  
  By Proposition~\ref{prop:isotoplift} the isotopy $H$ can be lifted
  to an isotopy $\widetilde H\: \CDach \times I\ra \CDach$ rel.\
  $\post(f)$ 
  such that $\widetilde H_0=\id_{\CDach}$ and  $H_t\circ f=f\circ
  \widetilde H_t$ for  $t\in I$.

  Since $\widetilde{\CC}\sub
  f^{-1}(\widetilde{\CC})$, it follows from Lemma
  \ref{lem:lifts_inverses} that $$\widetilde{H}_1(\widetilde{\CC})
  \subset \widetilde{H}_1(f^{-1}(\widetilde{\CC})) =
  f^{-1}(H_1(\widetilde{\CC})) = f^{-1}(\CC).$$
  This means that
  the Jordan curve $\CC'\coloneqq  \widetilde{H}_1(\widetilde{\CC})$ is contained
  in $f^{-1}(\CC)$. Moreover,  it contains all postcritical points,
  since $\widetilde{\CC}$ does,  and the points in $\post(f)$ stay fixed under the isotopy $\widetilde{H}$. As we have seen above, no such Jordan curve exists and we get a contradiction as desired.
 \end{proof}

By a similar (though somewhat  lengthier) argument one can show that
the Latt\`{e}s map $f(z)= \frac{\iu}{2}(z + 1/z)$ 
does not have an 
$f$-invariant Jordan curve $\CC$ with $\post(f)\subset \CC$ (this map was
considered in Example~\ref{ex:Lattes_Milnor}). Another such
example can be found in \cite[Section~4]{CFP10}. 

We now turn to the proof of the  necessary and sufficient criterion for the existence   of an invariant Jordan curve as formulated in 
Theorem~\ref{thm:exinvcurvef}.  Note that in condition \ref{item:ex_invC2} of this theorem the requirement  on $\widehat{f}$ is
meaningful. 
Indeed, $H_1$ is isotopic to
$\id_{S^2}$  rel.~$\post(f)$. By 
Lemma~\ref{lem:T-eq_crit_post} this implies that  $\widehat f=H_1
\circ f$ 
is a Thurston map with 
$\post(\widehat
f)=\post(f)$. 

Furthermore, $ \CC'$ is a Jordan curve with
 $\post(\widehat f)=\post(f)\sub \CC' $. Since $ \CC'=H_1(\CC)\subset f^{-1}(\CC)$, we have that 
$\widehat{f}( \CC')= (H_1\circ f) ( \CC')\subset H_1(\CC)=  \CC'$. Hence 
$ \CC'$ is invariant with respect to $\widehat{f}$, and it makes sense
to require that $\widehat{f}$ is combinatorially expanding for
$\CC'$ (note that $\#\post(\widehat{f})=\#\post(f)\ge 3$, because $f$ is expanding).

\begin{proof} [Proof of Theorem~\ref{thm:exinvcurvef}]
\ref{item:ex_invC1} $\Rightarrow$~\ref{item:ex_invC2} This implication is trivial. Indeed, suppose $\widetilde \CC$ 
is as in \ref{item:ex_invC1}. Then in \ref{item:ex_invC2} we  let  $\CC=\CC'= \widetilde \CC$, and the isotopy $H$ be such that $H_t=\id_{S^2}$ for all $t\in I$. Then
$ \CC'= \widetilde \CC\sub f^{-1}(\widetilde \CC)=f^{-1}(\CC)$, and  $\widehat f=f$ is   combinatorially expanding for the invariant curve  $ \CC'= \widetilde \CC$, because  $f$ is expanding.

\ref{item:ex_invC2} $\Rightarrow$~\ref{item:ex_invC1}   Let $\CC$, $\CC'$, $H$, 
$\widehat f$ be as in \ref{item:ex_invC2}, and define $\chi=H_1$. 
As we have seen in the discussion before the proof,  $\widehat f$ is  a Thurston map
with $\post (\widehat f)=\post(f)$, and $ \CC'$ is an $\widehat f$-invariant Jordan curve containing the set  $\post (\widehat f)$.

Since $\widehat f$  is combinatorially expanding for 
$ \CC'$,  Theorem~\ref{thm:combexp1} implies that there exists a homeomorphism $\phi\:S^2\ra S^2$ that is isotopic to the identity on $S^2$ rel.~$\post(\widehat f)=\post(f)$ such that 
$\phi( \CC')= \CC'$ and $g=\phi\circ \widehat f$ is an expanding Thurston map.
Since $g=(\phi\circ \chi)\circ f$, and $\phi\circ \chi$ is isotopic to the identity on $S^2$ rel.~$\post(f)$, the expanding Thurston maps $f$ and $g$ are Thurston equivalent: 
if notation is as in \eqref{Thequiv1} (with $\widehat S^2=S^2$), then we can take $h_0=\phi\circ \chi$ and $h_1=\id_{S^2}$.
By Theorem~\ref{thm:exppromequiv} we can find a homeomorphism 
$h\:S^2\ra S^2$ that is isotopic to $h_1=\id_{S^2}$  rel.~$f^{-1}(\post(f))$ with $h\circ f=g\circ h$. Note that then the homeomorphism $h^{-1}$ is also isotopic 
to $\id_{S^2}$  rel.~$f^{-1}(\post(f))$ and we have $f\circ h^{-1}=
h^{-1}\circ g$.   

Let $ \widetilde \CC=h^{-1} (\CC')$. Then 
$ \widetilde \CC$ is  a Jordan curve in $S^2$ that is isotopic to $ \CC'$  rel.~$f^{-1}(\post(f))$, and hence isotopic to $\CC$ rel.\ $\post(f)$; in particular,  $ \widetilde \CC$ contains the set $\post(f)$. Moreover,  $\widetilde \CC$ is $f$-invariant, because we have 
\begin{align*}
f(\widetilde  \CC)&=f(h^{-1}(\CC'))= h^{-1}(g( \CC'))=h^{-1}(\phi(\widehat f( \CC')))\\
&\sub 
h^{-1}(\phi( \CC'))=h^{-1}( \CC')= \widetilde \CC. 
\end{align*} 
The proof is complete.  \end{proof}

\begin{rem}
  \label{rem:expcheck}
  (i) Combinatorial expansion in condition
  \ref{item:ex_invC2} of Theorem~\ref{thm:exinvcurvef} is easy to
  check explicitly. A simple sufficient criterion for this can be
  formulated as follows: if no $1$-tile for $(f,\CC)$ joins
  opposite sides of $ \CC'$, then $\widehat f$ is combinatorially
  expanding for $ \CC'$.

  To see this, note that
$$\widehat f^{-1}( \CC')=f^{-1}(H_1^{-1}( \CC'))= f^{-1}(\CC).$$
By Proposition~\ref{prop:celldecomp}~\ref{item:nedgesC} this implies that the
$1$-tiles for $(\widehat f, \CC')$ are precisely the $1$-tiles
for $(f,\CC)$.  Hence if no $1$-tile for $(f,\CC)$ joins
opposite sides of $ \CC'$, then $D_1(\widehat f, \CC')\ge 2$ and
so $\widehat f$ is combinatorially expanding for $ \CC'$. We
will later formulate a necessary and sufficient condition for
combinatorial expansion of $\widehat f$ (see
Proposition~\ref{prop:nscombexp}; Example~\ref{ex:Cinv_notcexp}
illustrates the situation when this condition is not satisfied).

\smallskip 
(ii) Combinatorial expansion in condition
\ref{item:ex_invC2} of Theorem~\ref{thm:exinvcurvef} is
independent of the chosen isotopy $H$. Indeed, let 
$H^1,H^2\: S^2\times I\ra S^2$ be  two isotopies rel.\ $\post(f)$
with
$H^1_0=H^2_0=\id_{S^2}$ and $H^1_1(\CC)=H^2_1(\CC)=\CC'$. Then
$\widehat f_1=H^1_1\circ f$ is combinatorially expanding for
$\CC'$ if and only if $\widehat f_2=H^2_1\circ f$ is
combinatorially expanding for $\CC'$. This follows immediately
from Lemma~\ref{lem:cexp_Cinv} (with $f=\widehat f_1$,
$g=\widehat f_2$, $h_0=H^2_1\circ (H^1_1)^{-1}$, $h_1=\id_{S^2}$,
and $\CC=\CC'$).

\smallskip
(iii) Theorem~\ref{thm:exinvcurvef} can be slightly modified to give necessary and sufficient conditions for the existence of an invariant curve in a given isotopy class rel.\ $\post(f)$ or rel.\ $f^{-1}(\post(f))$. An  existence statement for  a given  isotopy class rel.\ $f^{-1}(\post(f))$ is especially relevant in view of the complementary uniqueness statement given by  Theorem~\ref{thm:uniqc}. 

To formulate this precisely,   let $\widehat \CC\sub S^2$ be a given Jordan curve  with $\post(f)\sub  \widehat \CC$.  Then an $f$-invariant Jordan curve $\widetilde \CC\sub S^2$    isotopic  to $\widehat \CC$ rel.\ $\post(f)$ exists if and only if condition \ref{item:ex_invC2} in Theorem~\ref{thm:exinvcurvef} is true 
for  a Jordan curve $\CC$  isotopic to $\widehat \CC$ rel.\ 
$\post(f)$. This immediately follows from the proof of this theorem.

 Similarly,  an $f$-invariant Jordan curve $\widetilde \CC\sub S^2$  isotopic  to $\widehat \CC$ rel.\ $f^{-1}(\post(f))$ exists if and only if condition \ref{item:ex_invC2} in  Theorem~\ref{thm:exinvcurvef} is true with   the extra assumption  that  $\CC'$ is isotopic to $\widehat \CC$ rel.\ $f^{-1}(\post(f))$.  
 \end{rem}

 The proof of the implication \ref{item:ex_invC2}
 $\Rightarrow$ \ref{item:ex_invC1} in 
 Theorem~\ref{thm:exinvcurvef}  does not only give the existence of
 an $f$-invariant Jordan curve $\widetilde{\CC}$,
 but $\widetilde{\CC}$ is constructed quite explicitly from the given
 Jordan curves $\CC$ and $\CC'$. So one can actually say more 
 about the combinatorial  description of $f$ in  terms of $\widetilde{\CC}$, or,  more
 precisely, about the two-tile
 subdivision rule that is given by $\widetilde{\CC}$ according to
 Proposition~\ref{prop:ThmapSub}.  Namely, 
 the $1$-tiles for $(f,\CC)$ subdivide the two $0$-tiles defined by
 $(f,\CC')$ in the same way as the $1$-tiles
 for $(f,\widetilde{\CC})$ subdivide the $0$-tiles for
 $(f,\widetilde{\CC})$. This is made precise in the following statement. 

\begin{cor}
  \label{cor:CtCpC_2tile_sub}
  Let $\widetilde{\CC}$ be the $f$-invariant Jordan curve
  obtained  in the proof of Theorem~\ref{thm:exinvcurvef} from
  the Jordan curves $\CC, \CC'\subset S^2$ as in condition
  \ref{item:ex_invC2} of  Theorem~\ref{thm:exinvcurvef}. Then
  there is a homeomorphism $h\colon S^2\to S^2$ that is isotopic
  to $\id_{S^2}$ rel.\ $f^{-1}(\post(f))$ with the
  following properties:
   $h(\post(f))= \post(f)$, $h(\widetilde{\CC}) = \CC'$,
   and   $h$ maps the $1$-cells  for $(f,\widetilde{\CC})$ to the $1$-cells  for $(f, \CC)$.
 \end{cor}

Note that the statement implies that $h$ also maps the $0$-cells for
$(f,\widetilde \CC)$ 
to the $0$-cells for $(f,\CC')$. The map $h$ is in fact
 the map that
appears in the proof of Theorem~\ref{thm:exinvcurvef}. 
Recall that an $n$-cell is an $n$-tile, $n$-edge, or a singleton
set $\{v\}$ where $v$ is an $n$-vertex.

To illustrate the statement, let us consider the map $f$ from
Example~\ref{ex:invC}.  
In the top right image of Figure~\ref{fig:invC_constr} we can
see that  one of the two $0$-tiles for $\CC^1 \subset f^{-1}(\CC)$ is subdivided 
into four   $1$-tiles for
$(f,\CC)$, and the other $0$-tile into two $1$-tiles for $(f,\CC)$. Thus the corollary above shows that the
two  $0$-tiles for the (unique) $f$-invariant curve $\widetilde{\CC}$
are 
subdivided into four or into two  $1$-tiles for
$(f,\widetilde{\CC})$.   

\begin{proof}
  We will  use the notation from the proof of the 
  implication
  \ref{item:ex_invC2} $\Rightarrow$~\ref{item:ex_invC1} 
 in  Theorem~\ref{thm:exinvcurvef}. So 
  we have  homeomorphisms $\chi, \phi, h\colon
  S^2\to S^2$ that satisfy
  \begin{equation*}
    \chi(\CC) = \CC', 
    \quad 
    \phi(\CC')= \CC', 
    \quad
    h(\widetilde{\CC}) = \CC'. 
  \end{equation*}
  The map $h$ is isotopic to $\id_{S^2}$ rel.\  $f^{-1}(\post(f))$, and so it fixes the points 
   in $f^{-1}(\post(f))\supset \post(f)$.
 
  The map $g= \phi \circ \chi \circ f$ is a  Thurston
  map conjugate  to  $f$  so that 
  $h\circ f = g\circ h$.  
  Lemma~\ref{lem:lifts_inverses} implies that 
    $h(f^{-1}(A)) = g^{-1}(h(A))$ for 
  each set $A\sub S^2$. Hence 
  \begin{align*} 
  h(f^{-1}(\widetilde \CC))&= g^{-1}(h(\widetilde \CC))=g^{-1}
  (\CC')
 =f^{-1}(\chi^{-1}(\phi^{-1}(\CC')))\\ &= f^{-1}(\chi^{-1}(\CC'))
  =f^{-1}(\CC). 
  \end{align*}
So by  Proposition~\ref{prop:celldecomp}~\ref{item:skeletons} the homeomorphism 
 $h$ maps the $1$-skeletons
  of the cell 
  decompositions $\DD^1(f,\widetilde \CC)$ and  $\DD^1(f, \CC)$
  (see Definition~\ref{def:DDn}) onto each other.
  The vertices of both cell decompositions are the points in 
  $f^{-1}(\post(f))$ which are fixed by $h$. Since these  cell decompositions are  uniquely determined by their  $1$-skeletons and  vertices (see Proposition~\ref{prop:celldecomp}~\ref{item:nedgesC}),  $h$ map the cells in  $\DD^1(f,\widetilde \CC)$  to the cells in   $\DD^1(f, \CC)$.
 Hence  $h$ maps $1$-cells for 
  $(f,\widetilde \CC)$ to $1$-cells for $(f,\CC)$. 
   \end{proof}

For the proof of Theorem~\ref{thm:main} we require the
following auxiliary result.  

\begin{lemma}
  \label{lem:isotopicpath} 
  Let $f\: S^2 \ra S^2$ be an expanding Thurston map,
  and $\CC\sub S^2$ be a Jordan curve with   $\post(f)\sub \CC$.
  Then for all sufficiently large $n$ there exists a Jordan
  curve  $\CC'\sub f^{-n}(\CC)$ that   is isotopic to $\CC$
  rel.\ $\post(f)$. Moreover, $\CC'$ can be chosen so that no
  $n$-tile for $(f,\CC)$ joins  opposite sides of  
  $\CC'$. 
\end{lemma} 
   
   \begin{proof}  We fix some  base metric on $S^2$.
    Let $P\coloneqq \post(f)$. Since $f$ is expanding, we have $k\coloneqq \#P=\#\post(f)\ge 3$
  by Lemma~\ref{lem:no<3}. Pick $\eps_0>0$ as in Lemma~\ref{lem:isorelP}. Since $f$ is expanding,  for large enough  $n$ we have
$$ \mesh(f,n,\CC)=\max_{c\in \DD^n(f,\CC)}\diam(c)<\eps_0.$$ 
For such $n$ consider the cell decomposition $\DD=\DD^n(f,\CC)$ of $S^2$. Its vertex set is the set $f^{-n}(\post(f))\supset  \post(f)=P$ of $n$-vertices  
and its $1$-skeleton is the set $f^{-n}(\CC)$.
Hence by Lemma~\ref{lem:isorelP} there exists a Jordan curve 
$\CC'\sub f^{-n}(\CC)$
that is  isotopic to $\CC$ rel.\ $P=\post(f)$ and so that no tile in $\DD$, i.e., 
no $n$-tile for $(f,\CC)$,  
joins opposite sides 
of $ \CC'$.
 \end{proof} 
   
   \begin{proof}[Proof of Theorem~\ref{thm:main}.]
Let $f$ and $\CC$ be as in the statement of the theorem.  By Lemma~\ref{lem:isotopicpath}  for sufficiently large  $n\in \N$ there exists  an 
isotopy $H\:S^2\times I\ra S^2$ rel.\ $\post(f)$ such that $H_0=\id_{S^2}$ and
$\CC'\coloneqq H_1(\CC)\sub f^{-n}(\CC)$ and   such that  no $n$-tile for $(f,\CC)$ joins opposite sides of $ \CC'$. 

If we define $F=f^n$ for such $n$, then the map $F$ is an expanding Thurston map with $\post(F)=\post(f)$. The sets  $\CC$ and  $\CC'$ are Jordan curves with $\post(F)\sub \CC,\CC'$,  and $H$ is an isotopy rel.\ $\post(F)$ that deforms $\CC$ into $\CC'\sub f^{-n}(\CC)=F^{-1}(\CC)$.  By  
Proposition~\ref{prop:celldecomp}~\ref{item:celdecompiter} the $1$-cells   
for $(F, \CC)$ are precisely the $n$-cells  for $(f,\CC)$.
So no $1$-tile for $(F,\CC)$ joins opposite side of 
$ \CC'$ and by Remark~\ref{rem:expcheck}~(i) the map $H_1\circ F$ is combinatorially expanding for $ \CC'$. 
This shows that condition \ref{item:ex_invC2} in 
Theorem~\ref{thm:exinvcurvef} is satisfied. Hence  there exists a Jordan curve $\widetilde \CC\sub S^2$ that is $F$-invariant and  isotopic to $\CC$ rel.\ $\post(F)=\post(f)$ as desired. 
 \end{proof}
 
 \begin{rem}\label{rem:ndep} In general, 
  the $f^n$-invariant Jordan curve $\widetilde{\CC}$  as in Theorem~\ref{thm:main} will 
  depend on $n$, and one cannot expect that $\widetilde{\CC}$ is invariant 
  for  {\em all}  sufficiently high iterates of $f$. To illustrate this, consider the map $f$ from 
  Example~\ref{ex:noinvCC} (see also Lemma~\ref{lem:g1noinvC}). Recall that $f(z)= \varphi(z^4)$ for $z\in \CDach$, where 
  $\varphi$ is as in \eqref{eq:mobius1}.
  
   The M\"{o}bius transformation
  $\varphi$ maps the extended real line $\widehat \R$ to  the unit circle $\partial \D$,  and $\partial \D$ to $\widehat \R$. 
  This implies that the unit circle
  $\widetilde \CC\coloneqq  \partial \D$ satisfies
  $f^{2n}(\widetilde \CC)\subset \widetilde \CC$ for every $n\in \N$. Note that $\post(f)=\{-\iu,1,\iu\}\subset
  \widetilde \CC$. Thus $\widetilde \CC$ is  a Jordan curve with
  $\post(f)\subset
  \widetilde \CC$   that is
  invariant for every \emph{even} iterate $f^{2n}$.

  On the other hand, for  $n\in \N_0$ we have  $f^{2n+1}(\partial \D)\subset \widehat \R$, and so we cannot 
  have  $f^{2n+1}(\partial \D)\sub \partial \D$ (for otherwise, $f^{2n+1}(\partial \D)\sub  \partial \D\cap \widehat \R=\{-1,1\}$). 
 Thus the unit circle $\partial \D=\widetilde \CC$ is not invariant for 
  any \emph{odd} iterate of $f$.  
\end{rem}

 \begin{proof}[Proof of Corollary~\ref{cor:subdivnlarge}] Let
   $f\:S^2\ra S^2$ be an expanding Thurston map. It follows from
   Theorem~\ref{thm:main} that for each sufficiently large $n\in \N$
   there exists an 
 $f^n$-invariant  Jordan curve $\widetilde \CC\sub S^2$ with
 $\post(f)=\post(f^n)\sub \widetilde \CC$. For such $n$ let
 $F=f^n$. By Proposition~\ref{prop:ThmapSub} there exists a two-tile subdivision rule that is realized by $F$.
 \end{proof}

\section{Iterative construction of  invariant curves}
\label{subsec:ittproc} 
 
Given data as in Theorem
\ref{thm:exinvcurvef} \ref{item:ex_invC2}, the  
$f$-invariant curve $\widetilde{\CC}$ can  be obtained by an iterative
procedure.  
To explain this,  let $f\: S^2\ra S^2$ be an arbitrary Thurston map, and assume as in Theorem~\ref{thm:exinvcurvef}~\ref{item:ex_invC2} that $\CC,\CC'\subset S^2$ are   Jordan curves with $\post(f)\subset \CC,\CC'$ and 
$\CC'\subset f^{-1}(\CC)$, and that $H\colon S^2\times
I\to S^2$ is an  isotopy rel.\ $\post(f)$ that deforms $\CC$ to
$\CC'$, i.e., $H_0=\id_{S^2}$ and  $H_1(\CC)=\CC'$. For the moment,  we
do \emph{not} assume that the map $f$ is expanding or that  $\widehat{f}= H_1\circ f$ is
combinatorially expanding for $\CC'$. 
 
Let $H^0\coloneqq  H$. 
By using   
Proposition~\ref{prop:isotoplift} repeatedly, we can find  isotopies  $H^n\colon S^2\times I\to S^2$   rel.\
$f^{-1}(\post(f))$ such that $H^n_0=\id_{S^2}$ and $f \circ H_t^{n+1} =
H_t^{n} \circ f$ for all $n\in \N_0$, $t\in I$. 
Now we define  Jordan curves  inductively by setting 
$\CC^0\coloneqq  \CC$, and  $\CC^{n+1}\coloneqq  H^n_1(\CC^n)$ for 
$n\in\N_0$.  Note that then $\CC^1=\CC'$. 

To summarize,  we start with  the following data for our given Thurston map $f$: 
\begin{enumerate}
  \item 
    \label{item:data_contr_Cinv1}
    A Jordan curve $\CC^0=\CC\sub S^2$ with  $\post(f)\subset \CC^0$.
 
\item 
    \label{item:data_contr_Cinv2}
  A  Jordan curve $\CC^1=\CC'\sub S^2$ isotopic to $\CC^0\sub S^2$ rel.\
    $\post(f)$ with $\CC^1\sub f^{-1}(\CC^0)$.
 
  \item
    \label{item:data_contr_Cinv3}
An  isotopy $H^0\:S^2\times I\ra S^2$ rel.\  $\post(f)$ such that
$H^0_0=\id_{S^2}$ and  $H^0_1(\CC^0)=\CC^1$.
\end{enumerate} 
 
 
  We then define inductively:
  \begin{enumerate}
  \item
    \label{item:constr_C_inv1}  
    Isotopies  $H^n\:S^2\times I\ra S^2$ such that $H^n_0=\id_{S^2}$ and $f \circ H_t^{n+1} =
H_t^{n} \circ f$ for all $n\in \N_0$, $t\in I$. 
  
  \item
    \label{item:constr_C_inv2}  
 Jordan curves   $ \CC^{n+1}\coloneqq    H^n_1(\CC^n)$  for $n\in \N_0$.
\end{enumerate}
  

Figure~\ref{fig:invC_constr} illustrates this procedure for 
 Example \ref{ex:invC}. Since this example is rather complicated and it is hard to grasp the  isotopies involved, we  present a simpler  example 
 for the construction. 
 
\ifthenelse{\boolean{nofigures}}{}{
  \begin{figure}
    \centering
    \begin{overpic}
    [width=12cm, 
    tics=20]{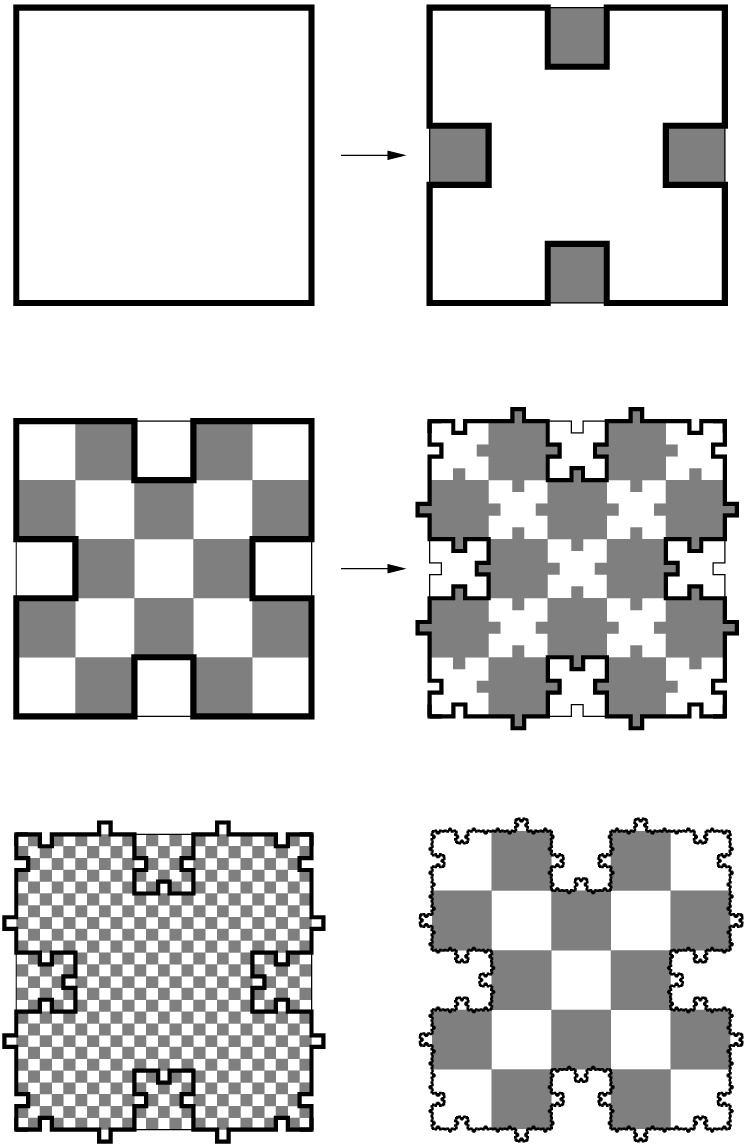}
    \put(31,87.8){$H^0_1$}
    \put(31,51.8){$H^1_1$}
    \put(31,14){$\dots$}
    \put(4,70){$\CC^0$}
    \put(39,70){$\CC^1$}
    \put(4,34){$\CC^1$}
    \put(39,34){$\CC^2$}
    \put(4,-2){$\CC^2$}
    \put(39,-2){$\widetilde{\CC}$}
    \end{overpic}
    \caption{Iterative construction of an invariant curve.}
    \label{fig:Cit}
  \end{figure}
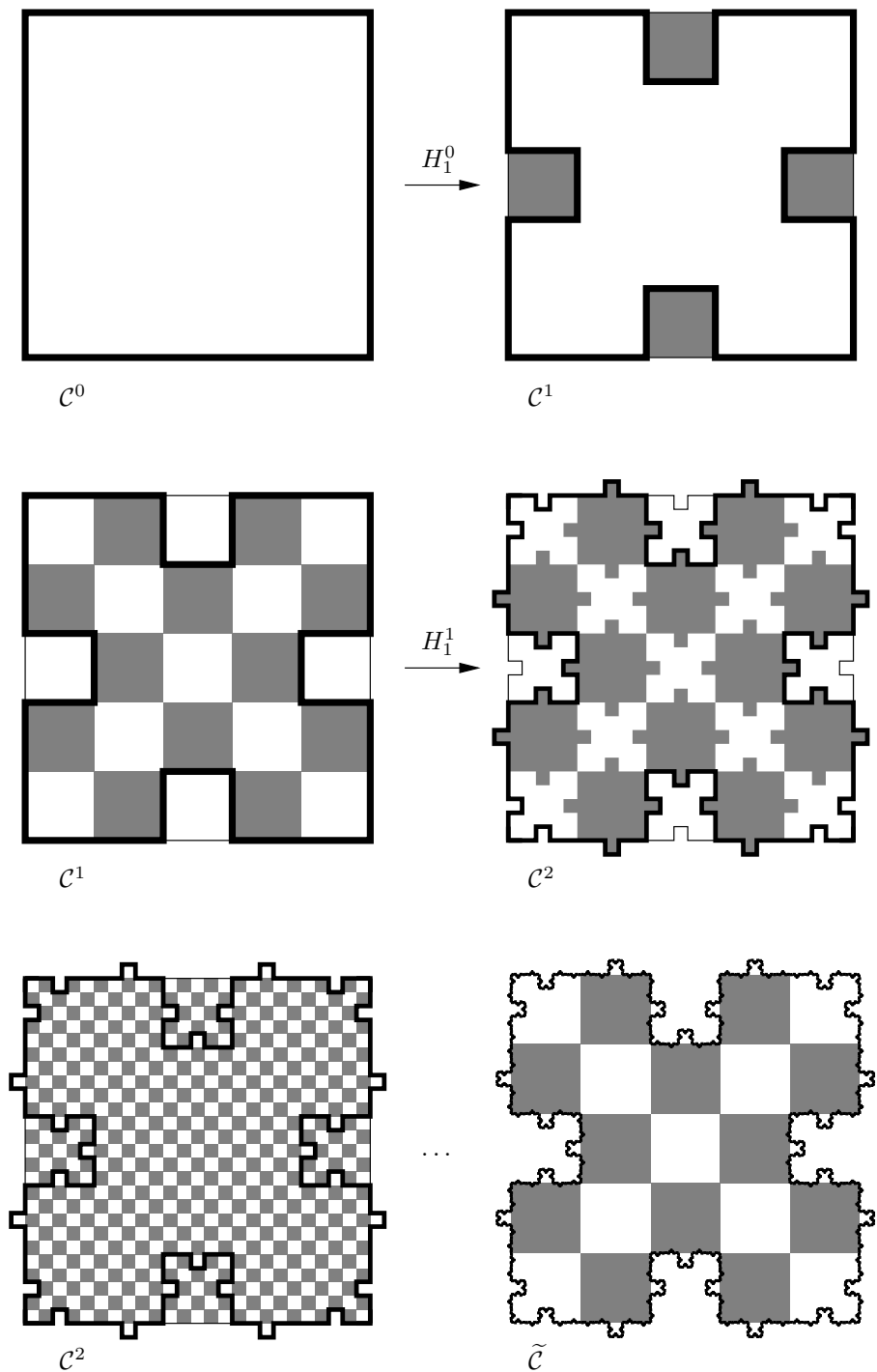
}

\begin{ex}
  \label{ex:Cit} 
  Let $f\colon \CDach \to \CDach$ be a  Latt\`{e}s map 
 constructed as   the example $g$ in Section~\ref{sec:Lattes}, but with the
  map   
  \begin{equation*}
    A\colon \C\to \C, 
    \quad 
    u \mapsto A(u)\coloneqq 5u.
  \end{equation*}
  More precisely, $f$ is obtained according to 
  Theorem~\ref{thm:Lattesstruc}~\ref{item:Lattessruciii} as the
  quotient of $A$ by a crystallographic group of type $(2222)$ as
  in \eqref{eq:specisom}.
  %
It is straightforward  to check that the extended real line
$\CC\coloneqq \widehat{\R}=\R\cup\{\infty\}$ is $f$-invariant and contains all
postcritical points $0,1,\infty,-1$ of $f$. 

As in Figure \ref{fig:mapg}, we represent the sphere $\CDach$ as a
pillow, i.e., two squares glued together along their boundaries. The
equator  of the pillow  represents the curve $\CC$, and the two
squares represent the $0$-tiles, one of which is colored white, the
other black. 

The map $f$ can then be  described  as follows. Each
of the two sides of the pillow is divided into $5\times 5$ squares,
which are colored in a checkerboard fashion. The map $f$
 sends each small white square to the white side of the pillow, and each
small black square to the black side. 
The two sides of the
pillow are the $0$-tiles for  $(f,\CC)$;  the $4$ vertices of
the pillow are the postcritical points in this model. 
The small
squares are the $1$-tiles for $(f,\CC)$. The coloring  of the $0$-
and $1$-tiles corresponds to a labeling map $L_\X$  as in 
Lemma~\ref{lem:labelexis}.

There exist  $f$-invariant Jordan curves that are  isotopic to $\CC$ 
rel.\ $\post(f)$, but distinct from $\CC$. The construction of one
such curve is illustrated in Figure~\ref{fig:Cit}. Namely,  we set
$\CC^0\coloneqq  \CC$. The Jordan curve $\CC^1$ is shown on the top right, as
well as 
in the middle left picture. In the latter picture, we see that $\CC^1$
consists of $1$-edges, i.e., $\CC^1\subset f^{-1}(\CC^0)$. Moreover, there
exists an isotopy $H^0\colon \CDach\times I\to \CDach$ rel.\
$\post(f)$ that 
deforms $\CC^0$ to $\CC^1$ (i.e., $H^0_0= \id_{\CDach}$ and 
$H^0_1(\CC^0)=\CC^1$). We also see here
how the black and the  white $0$-tile are deformed by $H^0_1$;
namely, the four  
small black squares on the top right in 
Figure \ref{fig:Cit} are part of the image of the black $0$-tile
(which is at the back of the pillow) under $H^0_1$. 

The Jordan curve $\CC^2\coloneqq  H^1_1(\CC^1)$ consists of $2$-edges, i.e., 
$\CC^2\subset f^{-2}(\CC^0)$ (see the bottom left). 
The two pictures in the middle of Figure \ref{fig:Cit} indicate
how $H^1$ deforms $1$-tiles. Roughly speaking, $H^1$ deforms each
black or white $1$-tile ``in the same way'' as $H^0$ deforms the
black or white $0$-tiles.

The curves $\CC^n$
Hausdorff converge to $\widetilde{\CC}$, which is an $f$-invariant
Jordan curve with $\post(f)\subset \widetilde{\CC}$ (see
Lemma~\ref{lem:CCnprop}~\ref{CCprop7} and 
Proposition~\ref{prop:invCit}). 
\end{ex}

There is a conceptually different  way to obtain $\CC^{n+1}$ from
$\CC^n$, which will be explained in detail in Remark
\ref{rem:C_arc_replace}. Namely,   we replace each $n$-edge
$\alpha^n\subset \CC^n$ with  $(n+1)$-edges ``in the same way'' as
the $0$-edge $\alpha^0\coloneqq  f^n(\alpha^n)\subset \CC^0$ is replaced with 
an  arc $\beta^1\subset \CC^1$ that has  the same
endpoints  as $\alpha^0$ (which are postcritical points).  Note that
$\beta^1=H^0_1(\alpha^0)$, and that $\beta^1$ consists of  $1$-edges.

To prepare the proof that under suitable conditions our iteration
process has an invariant  curve as a  limit in the sense of Hausdorff convergence,  
we summarize some properties of the Jordan curves $\CC^n$.

\begin{lemma}  
  \label{lem:CCnprop}
  Let $f\colon S^2\ra  S^2$ be a Thurston map 
  that satisfies
  $\#\post(f)\ge 3$, 
  and let  the Jordan curves $\CC^n$ for $n\in
  \N_0$ be defined as above. Then the following statements are true:
  
  \begin{enumerate}
  
  \item    \label{CCprop1}
    $ \CC^{n+k}\subset f^{-k}(\CC^{n})$  for
    $n,k\in \N_0$.

  \item   \label{CCprop2}   
    $\CC^{n+k}$ is isotopic to $
    \CC^n$ rel.\ $f^{-n}(\post(f))$ for $n,k\in \N_0$.

  \item   \label{CCprop3}   
    $\CC^{n+k}\cap f^{-n}(\post(f))=
    \CC^n\cap f^{-n}(\post(f))$ for $n,k\in \N_0$.    
   
  \item   \label{CCprop4}   
    $\post(f)\sub \CC^n$ for $n\in \N_0$.  
    
  \item    \label{CCprop4a}
    For  $n,k\in \N_0$ the curve  $ \CC^{n+k}$ consists of $n$-edges
    for $(f, \CC^k)$.

  \item   \label{CCprop5}   
   For $n\in \N$ the curve $\CC^n$  is the unique Jordan curve in $S^2$ 
   with $\CC^n\sub f^{-1}(\CC^{n-1})$ that is isotopic to $\CC^1$ rel.\ 
   $f^{-1}(\post(f))$.   
   
  \item   \label{CCprop6}   
    The sequence $\CC^n$, $n\in \N_0$, only depends on $\CC^0$ and
    $\CC^1$ and not on  the  choice of the initial isotopy $H=H^0$ used
    in the definition of the sequence.    
    
  \item   \label{CCprop7}   Suppose in addition that $f$  is expanding. 
   Then, as $n\to \infty$,  the sets  $\CC^n$  Hausdorff converge     to a closed
    $f$-invariant set $\widetilde \CC\sub S^2$  with $\post(f)\sub \widetilde \CC$.     
  \end{enumerate}
\end{lemma}

Recall that Hausdorff convergence was discussed at the end of 
Section~\ref{sec:QCgeom}. 

\begin{proof} In the following, we use the isotopies $H^n$ as in the definition of the sequence $\CC^n$, and set $h_n=H^n_1$ for $n\in \N_0$. 

\smallskip
\ref{CCprop1}  
It suffices to show that $\CC^n \sub f^{-1}(\CC^{n-1})$ for $n\in \N$. We prove  this by induction on $n$; this   is clear for $n=1$. Assume that
  the statement holds for some $n\in \N$;  so  $\CC^n\subset
  f^{-1}(\CC^{n-1})$. Since  $h_{n}=H^n_1$  and $h_{n-1}=H^{n-1}_1$ are homeomorphisms with $f\circ h_{n}=h_{n-1}\circ f$, 
  we have  $h_n(f^{-1}(\CC^{n-1}))= f^{-1}(h_{n-1}(\CC^{n-1}))$
  by Lemma~\ref{lem:lifts_inverses}. 

\pagebreak
Thus 
  \begin{equation*}
    \CC^{n+1}= h_n(\CC^n) \subset h_n(f^{-1}(\CC^{n-1}))=
    f^{-1}(h_{n-1}(\CC^{n-1}))= f^{-1}(\CC^n), 
  \end{equation*}
  and \ref{CCprop1} follows.
 
 \smallskip
 \ref{CCprop2}--\ref{CCprop4}
  From  the definition of $H^n$, the remark after the proof of
  Proposition~\ref{prop:isotoplift}, and  
induction on $n$, we conclude 
that $H^n$ 
  is an isotopy rel.\ $f^{-n}(\post(f))$. Since $H^n_0=\id_{S^2}$ and $$f^{-n}(\post(f))\sub f^{-(n+k)}(\post(f)) $$ for $n,k\in \N_0$,  
 statements \ref{CCprop2} and \ref{CCprop3} immediately 
  follow from this by induction on $k$ for fixed $n$. Statement \ref{CCprop4} follows from \ref{CCprop3} (with $n=0$ and $k\in \N_0$ arbitrary) and the fact that $\post(f)\sub \CC^0$. 
  
  \smallskip 
  \ref{CCprop4a}
  By \ref{CCprop4} we have $\#(f^{-n}(\post(f))\cap \CC^{n+k})\ge
\#\post(f)\ge 3.$ In particular, the points in $f^{-n}(\post(f))$ that
lie on $\CC^{n+k}$ 
subdivide this curve into arcs whose endpoints lie in $f^{-n}(\post(f))$ and whose interiors are disjoint from $f^{-n}(\post(f))$.
Let $\alpha\sub \CC^{n+k}$ be one of these arcs. Then 
we have $\inte(\alpha)\sub f^{-n}(\CC^k)\setminus f^{-n}(\post(f))$ by \ref{CCprop1}, and 
$\partial \alpha \sub f^{-n}(\post(f))$. Since by Proposition~\ref{prop:celldecomp}~\ref{item:skeletons} the set $f^{-n}(\CC^k)$ is the $1$-skeleton and the 
set $f^{-n}(\post(f))$ the $0$-skeleton of the cell decomposition $\DD^n(f, \CC^k)$, we conclude from Lemmas~\ref{lem:conncomp} and \ref{lem:opencells} that $\alpha$ is an edge in  $\DD^n(f, \CC^k)$, i.e., an $n$-edge for $(f,\CC^k)$. Hence $\CC^{n+k}$ consists of $n$-edges for 
$(f,\CC^k)$.  


\smallskip   
\ref{CCprop5}
By \ref{CCprop1} and \ref{CCprop2} we know that $\CC^n$ for $n\in \N$ is a Jordan curve 
 with $\CC^n\sub f^{-1}(\CC^{n-1})$ that is  isotopic to $\CC^1$ rel.\ $f^{-1}(\post(f))$. Let  $\widehat  \CC\sub  f^{-1}(\CC^{n-1})$ be another Jordan curve isotopic to $\CC^1$ rel.\ $f^{-1}(\post(f))$. Then $\CC^n$ and $\widehat \CC$ are isotopic to each other rel.\ $f^{-1}(\post(f))$. Note that $f^{-1}(\CC^{n-1})$ is the $1$-skeleton of the cell decomposition $\DD^1(f, \CC^{n-1})$ and $f^{-1}(\post(f))$ is its set of vertices. Since each tile in $\DD^1(f, \CC^{n-1})$ has at least $\#\post(f)\ge 3$ vertices, we can apply
   Lemma~\ref{lem:isoJcin1ske} and conclude that $\widehat \CC=\CC^n$. The uniqueness statement for  $\CC^n$  follows. 

\smallskip
\ref{CCprop6}
It follows from \ref{CCprop5} and induction on $n$ that $\CC^n$ is uniquely determined  by  $\CC^0$ and $\CC^1$. 

\smallskip
\ref{CCprop7} Since $f$ is expanding, we can 
pick a visual metric $\varrho$ for $f$.  Let $\Lambda >1$ be
the expansion factor of $\varrho$. By Lemma~\ref{lem:exp_shrink} the
diameters of the tracks
of the isotopy $H^n$ are bounded by $C\Lambda^{-n}$, where $C$ is a
fixed constant. Since $H^n_0=\id_{S^2}$ and $\CC^{n+1}=H^n_1(\CC^n)$
for  
$n\in \N_0$, this implies that $\dist_\varrho^H(\CC^n,\CC^{n+1})\le
C\Lambda^{-n}$ for $n\in \N_0$. It follows that the 
sequence $\{\CC^n\}$
is a Cauchy sequence with respect to Hausdorff distance. 
Recall that the space
of all non-empty closed subsets of  a compact metric space is complete 
if it is equipped with the Hausdorff distance.
 Thus there exists a non-empty closed set $\widetilde \CC\sub S^2$ such that  $\CC^n\to \widetilde \CC$ as $n\to \infty$ in the sense of Hausdorff convergence. Since $\post(f)\sub \CC^n$ for all $n\in \N_0$ by \ref{CCprop4}, we have $\post(f)\sub \widetilde \CC$. 

It remains to show that $\widetilde \CC$ is $f$-invariant. To see this, let $p\in \widetilde \CC$ be arbitrary. Then there exists a sequence $\{p_n\}$
of points in $S^2$ such that $p_n\in \CC^n$ for $n\in \N_0$ and $p_n\to p$ as $n\to \infty$. By continuity of $f$ we have $f(p_n)\to f(p) $ as $n\to \infty$. Moreover, \ref{CCprop1} implies that $f(p_n)\in \CC^{n-1}$ for $n\in \N$. Hence $f(p)\in \widetilde\CC$, and so the set $\widetilde \CC$ is indeed $f$-invariant. 
\end{proof}

As an application of the preceding setup  we prove a statement 
 that gives a necessary and sufficient condition for the map $\widehat{f}$ 
 in Theorem~\ref{thm:exinvcurvef} to be combinatorially expanding.
 
 \begin{prop}\label{prop:nscombexp}
   Let $f\:S^2\ra S^2$ be a Thurston map with $\#\post(f)\ge 3$,
   and let the isotopy $H^0\:S^2\times I\ra S^2$ and Jordan curves $\CC^n$ for $n\in \N_0$ be defined as above. 
   
 Then $\widehat{f}=H^0_1 \circ f$ is
  combinatorially expanding for $\CC^1=\CC'$ if and only if there exists $n\in \N$ such
  that no $n$-tile for  $(f,\CC^0)$ joins  opposite sides
  of $\CC^n$. 
  \end{prop}

 \begin{proof}   Let $H^n$ for  $n\in \N_0$ be the isotopies used in the definition of the curves $\CC^n$. Set $h_n\coloneqq  H_1^n$. Then $\widehat f=
h_0\circ f$, $\CC^{n+1}=h_n(\CC^n)$,  and  $h_n\circ f = f
  \circ h_{n+1}$ for  $n\in \N_0$.  It follows by induction that for $n\in \N$ we have 
   $$
    \widehat{f}^n =h_0\circ f \circ \dots \circ h_0\circ f= h_0 \circ f^n \circ h_{n-1} \circ \dots \circ h_1, $$
    and so
    $$  h_0 \circ f^n= \widehat{f}^n\circ h_1^{-1}\circ \dots
    \circ h_{n-1}^{-1} . $$
    Hence 
  $$  f^{-n}(\CC^0)=f^{-n}(h_0^{-1}(\CC^1))=
    ( h_{n-1}\circ \dots \circ h_1) ( \widehat{f}^{-n}(\CC^1)). $$

Recall that the $n$-tiles for  $(f,\CC^0)$ are the closures of the
complementary components of $f^{-n}(\CC^0)$, and the $n$-tiles for $(\widehat f, \CC^1)$  the closures of the
complementary components of $\widehat {f}^{-n}(\CC^1)$ (Proposition \ref{prop:celldecomp}~\ref{item:nedgesC}).  So from the previous identity we conclude  that the
  $n$-tiles for  $(f,\CC^0)$ are precisely the images of
  the $n$-tiles  for  $(\widehat{f},\CC^1)$ under the
  homeomorphism $h_{n-1}\circ \dots \circ h_1$. Note that this
  homeomorphism is isotopic to $\id_{S^2}$ rel.  $\post(f)=\post(\widehat f)$ and maps $\CC^1$ to $\CC^n$. Thus  no  
  $n$-tile for  $(\widehat{f},\CC^1)$ joins   opposite
  sides of $\CC^1$ if and only if no  $n$-tile for $(f,\CC^0)$ joins   opposite sides of $\CC^n$. 

  Now $\widehat{f}$ is combinatorially expanding for $\CC^1$ if and only if 
  there exists $n\in \N$ such that  no $n$-tile for   $(\widehat{f}, \CC^1)$ joins  opposite sides of $\CC^1$. By what we have seen, this is the case if and only if there exists $n\in \N$ such that no $n$-tile for $(f,\CC^0)$
  joins  opposite sides of $\CC^n$. 
  \end{proof}

Let us now assume that our Thurston map $f$ is expanding. 
Then the curves $\CC^n$ Hausdorff converge to an $f$-invariant
closed set $\widetilde{\CC}$ by
Lem\-ma~\ref{lem:CCnprop}~\ref{CCprop7}. 
In general, $\widetilde \CC$ will not be a
Jordan curve (see Example~\ref{ex:Cinv_notcexp}). The following proposition shows that $\widetilde \CC$ is a Jordan curve if  the map $\widehat{f}=H^0_1 \circ f$ is combinatorially expanding for $\CC^1$.  Actually, one can  show that this condition is also  necessary for  $\widetilde \CC$ to be a Jordan curve, but we will not present the proof for this statement  as it is somewhat involved.

\begin{prop}[Iterative procedure for invariant curves]
  \label{prop:invCit}
   Let $f\:S^2\ra S^2$ be an expanding Thurston map,  and suppose the isotopy $H^0\:S^2\times I\ra S^2$ and Jordan curves $\CC^n$ for $n\in \N_0$ are  defined as above. 
   
   If $\widehat{f}=H^0_1 \circ f$ is combinatorially expanding for
      $\CC^1=\CC'$,  then 
$\CC^n$  Hausdorff converges   to a
       Jordan curve $\widetilde \CC\sub S^2$ as  $n\to \infty$. 
    In this case, the curve  $\widetilde \CC$ is $f$-invariant 
and $\post(f)\sub \widetilde \CC$. Moreover, $\widetilde{\CC}$ is
isotopic to $\CC^1$ rel.\ $f^{-1}(\post(f))$. 
\end{prop}
By Theorem~\ref{thm:uniqc} the curve $\widetilde{\CC}$ is  the unique Jordan curve with the given properties.

\begin{proof}   Suppose that $\widehat f=H^0_1\circ f$ is combinatorially expanding for $\CC^1$.  From Theorem
  \ref{thm:exinvcurvef} it follows that there exists an $f$-invariant
  Jordan curve $\widetilde{\CC}\sub S^2$ with $\post(f)\sub \widetilde \CC$ that is isotopic to $\CC^0$ rel.\ $\post(f)$ and isotopic to  $\CC^1$ rel.\
  $f^{-1}(\post(f))$. Let $K^0\colon S^2\times I\to S^2$ be an isotopy
  rel.\ $\post(f)$ that deforms $\widetilde{\CC}$ to $\CC^0$; 
  so $K^0_0=\id_{S^2}$ and  $K^0_1(\widetilde{\CC})= \CC^0$. Using 
  Proposition \ref{prop:isotoplift} repeatedly, we can find isotopies $K^n\: S^2\times I\ra S^2$ rel.\ $f^{-1}(\post(f))$ with $K^n_0=\id_{S^2}$ such that 
  $f\circ K^n_1=K^{n-1}_1\circ f$ for $n\in \N$. 

  \smallskip
  {\em Claim.} $ \widetilde \CC^n\coloneqq  K^n_1(\widetilde
  \CC)=\CC^n$ for all $n\in \N_0$.

\smallskip 
  We prove this claim by induction on $n$; it follows from the choice of $K^0$   for
  $n=0$. Assume that the statement is true  for some $n\in \N_0$. Then 
 $K_1^{n}(\widetilde \CC)= \CC^n$, and so by Lemma
  \ref{lem:lifts_inverses} we  have 
 $$\widetilde \CC^{n+1}=K_1^{n+1}(\widetilde \CC)\sub
 K_1^{n+1}(f^{-1}(\widetilde \CC))=f^{-1}( K_1^{n}(\widetilde \CC))=
 f^{-1}(\CC^n).$$
   Since $K^{n+1}$ is  an isotopy rel.\
  $f^{-1}(\post(f))$, the curve $\widetilde \CC^{n+1}$ is isotopic to
  $\widetilde \CC$ and hence to $\CC^1$ rel.\  $f^{-1}(\post(f))$. So
  Lemma \ref{lem:CCnprop}~\ref{CCprop5} implies that 
  $\widetilde \CC^{n+1}=\CC^{n+1}$. This proves the  claim. 

\smallskip
It follows from    Lemma \ref{lem:exp_shrink} that  the maps $K^n_1$ converge  uniformly to the
  identity on $S^2$ as $n\to \infty$ . Hence $\CC^n=K_1^n(\widetilde{\CC})$  Hausdorff
 converges to the Jordan curve  $\widetilde{\CC}$ as $n\to \infty$. 
 The  statement follows. 
\end{proof}

\begin{rem}
  \label{rem:iterate_given_Ct} If $f\: S^2\ra S^2$ is an expanding Thurston map, then 
  every $f$-invariant Jordan curve $\widetilde{\CC}$ with
  $\post(f)\subset \widetilde{\CC}$ can  be obtained by our  iterative
  procedure. Indeed, suppose  that  $\widetilde{\CC}$ is such a curve. 
Trivially,   we can then  take $\CC=\CC^0=\widetilde \CC$,  $\CC'=\CC^1=\widetilde \CC$, and $H^0_t=\id_{S^2}$ for $t\in I$. Then $\CC^n=\widetilde
  \CC$ for all $n\in \N_0$  and so  $\CC^n \to \widetilde \CC$ as $n\to \infty$. 
  
  Actually, a  much stronger statement is true. Namely, we can start 
  with {\em any} Jordan curve $\CC$ in the same isotopy class rel.\ 
  $\post(f)$ as $\widetilde \CC$. Suppose that   $\CC$ is   such a curve. First, we  claim that then there exists a  unique Jordan curve $\CC'\sub f^{-1}(\CC)$ that is isotopic to $\widetilde \CC$ rel.\ $f^{-1}(\post(f))$.  To see this, let $K^0\colon S^2\times I \ra 
  S^2$ be an isotopy rel.\ $\post(f)$ with $K^0_0=\id_{S^2}$ and 
   $K^0_1(\widetilde{\CC})=\CC$. 
 By  Proposition
  \ref{prop:isotoplift} we can  lift $K^0$ by $f$ to an isotopy $K^1$ rel.\ $f^{-1}(\post(f))$  with  $K^1_0=\id_{S^2}$ and $K^0_t
  \circ f = f\circ K^1_t$ for $t\in I$.   
Then the Jordan curve $\CC': = K^1_1(\widetilde{\CC})$ satisfies
  \begin{equation*}
    \CC'= K^1_1(\widetilde{\CC}) \subset
    K^1_1(f^{-1}(\widetilde{\CC})) = f^{-1}(K^0_1(\widetilde{\CC})) =
    f^{-1}(\CC). 
  \end{equation*}
  Here we used $\widetilde{\CC}\subset f^{-1}(\widetilde{\CC})$ and Lemma \ref{lem:lifts_inverses}.  This shows existence of a curve $\CC'$ with the desired properties. Uniqueness of $\CC'$ follows from
   Lemma~\ref{lem:isoJcin1ske} (applied to $\DD=\DD^1(f,\CC)$). 
   
  Define  $H\: S^2\times I\ra S^2$ by setting 
 $H_t=K^1_t\circ (K^0_t)^{-1}$ for $t\in I$. Then $H$ is an isotopy rel.\ $\post(f)$ that deforms $\CC^0\coloneqq \CC$ into 
 $\CC^1\coloneqq \CC'$. Indeed, we have  $ H_0=\id_{S^2}$ and 
 $$H_1(\CC^0)=K^1_1( (K^0_1)^{-1}(\CC))=K^1_1(\widetilde \CC)=\CC'=\CC^1. $$   
   Moreover,   
   $$\widehat{f}\coloneqq  H_1\circ f=
  K^1_1\circ (K^0_1)^{-1} \circ f = K^1_1 \circ f \circ
  (K^1_1)^{-1}. $$ Thus it follows from Lemma \ref{lem:cexp_Cinv} that
  $\widehat{f}$ is combinatorially expanding for $\CC^1=\CC'=K^1_1(\widetilde \CC)$. 

  Define the sequence $\{\CC^n\}$ starting from $\CC^0$ and $ \CC^1$ as
  before. From Proposition \ref{prop:invCit} it follows that as $n\to \infty$ the curves $\CC^n$ Hausdorff
  converge to an $f$-invariant Jordan curve that is isotopic to $\CC^1$, and 
  hence isotopic to $\widetilde{\CC}$,  rel.\
  $f^{-1}(\post(f))$. From Theorem \ref{thm:uniqc} it follows that the
  unique such curve is $\widetilde{\CC}$. Thus $\CC^n\to\widetilde{\CC}$ in the   Hausdorff sense as $n\to \infty$.  
\end{rem}

\begin{rem}\label{rem:C_arc_replace}
Let $f\: S^2\ra S^2$ be a Thurston map with $\post(f)\ge 3$. Then
in the inductive definition of 
  $\CC^{n+1}=H^n_1(\CC^n)$ one  
 can  construct $\CC^{n+1}$  from $\CC^n$ by an   {\em edge replacement
   procedure}  
without explicitly knowing the isotopy  
$H^n$. To explain this, suppose that $n\in \N$,  and that $\CC^n$ has
already been constructed (starting from given curves $\CC^0$ and
$\CC^1$). We know by Lemma~\ref{lem:CCnprop}~\ref{CCprop4a} 
that $\CC^n$ consists of $n$-edges $\alpha^n$  for $(f,\CC^0)$.  
  Then 
$\CC^{n+1}$ is obtained from $\CC^n$ by replacing 
each  $n$-edge $\alpha^n\sub \CC^n$ with  a certain arc  $\beta^{n+1}$
with the same endpoints as $\alpha^n$. 

Indeed, we can set  $\beta^{n+1}\coloneqq H^n_1(\alpha^n)\sub \CC^{n+1}$. Then the union of these arcs $\beta^{n+1}$ is equal to $\CC^{n+1}$. Moreover, since $H^n$ is an isotopy relative  to the set $f^{-n}(\post(f))$ of $n$-vertices,  and $\alpha^n$ is an $n$-edge for $(f,\CC^0)$ and so has $n$-vertices as endpoints, the arcs $\alpha^n$ and $\beta^{n+1}$ have the same endpoints.

Now the arc  $\beta^{n+1}$ is the unique arc in  $f^{-n}(\CC^1)$ that is isotopic 
  to $\alpha^n$ rel.\ $f^{-n}(\post(f))$. This property   often allows
  one to determine $\beta^{n+1}$  directly from $\alpha^n$.

To see that this  characterization of $\beta^{n+1}$ holds, note that by Lem\-ma~\ref{lem:CCnprop}~\ref{CCprop1} we have $\beta^{n+1}\sub \CC^{n+1}\sub f^{-n}(\CC^1)$. 
Moreover, $\beta^{n+1}=H^n_1(\alpha^n)$ is isotopic to $\alpha^n$ rel.\ $f^{-n}(\post(f))$.

Suppose $\widetilde \beta^{n+1}\sub f^{-n}(\CC^1)$ is another  arc
that is isotopic to $\alpha^n$ rel.\ $f^{-n}(\post(f))$.  
Then the arcs $\beta^{n+1}$ and  $\widetilde \beta^{n+1}$ have  endpoints in 
$f^{-n}(\post(f))$, but contain no other points in this set, since
this is true for $\alpha^n$. This and the inclusions
$\beta^{n+1},\widetilde \beta^{n+1}\sub f^{-n}(\CC^1)$ imply that  
$\beta^{n+1}$ and $\widetilde \beta^{n+1}$ are  $n$-edges for $(f, \CC^1)$
(see the argument in the proof of Lemma~\ref{lem:CCnprop}~\ref{CCprop4a}).   
Since 
 $\beta^{n+1}$ and $\widetilde \beta^{n+1}$ are isotopic relative to the set $f^{-n}(\post(f))$,  which is the $0$-skeleton of $\DD^n(f, \CC^1)$, it follows from the first part of the proof of Lemma~\ref{lem:isoJcin1ske} that $\beta^{n+1}=\widetilde 
 \beta^{n+1}$ as desired.

As we have just seen, $\beta^{n+1}$ is an $n$-edge for $(f,\CC^1)$.
Since $\beta^{n+1}$ has endpoints in the set 
$f^{-n}(\post(f))\sub f^{-(n+1)}(\post(f))$ and  $\beta^{n+1}\sub\CC^{n+1}\sub 
f^{-(n+1)}(\CC^0)$, a  similar argument also shows that $\beta^{n+1}$ consists of $(n+1)$-edges for $(f,\CC^0)$.

\ifthenelse{\boolean{nofigures}}{}{
  \begin{figure}
    \centering
    \begin{overpic}
      [width=12cm, 
      tics=20]{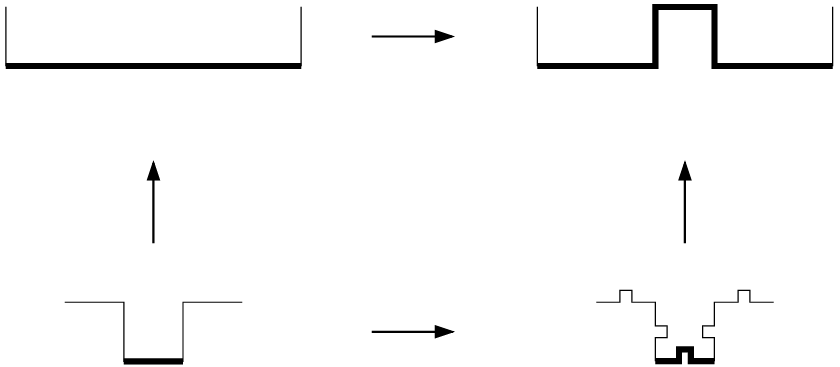}
      \put(14,31){$\alpha^0\subset \CC^0$}   
      \put(47,41){$H^0_1$}   
      \put(47,6){$H^n_1$}
      \put(20,18.5){$f^n|\alpha^n$} 
      \put(17,-3){$\alpha^n\subset\CC^n$}
      \put(79,-3.4){$\beta^{n+1}\subset \CC^{n+1}$}
      \put(83,18.5){$f^n|\beta^{n+1}$}
      \put(76,31){$\beta^1\subset\CC^1$}                 
    \end{overpic} 
    \caption{Iterative construction by replacing edges.}
    \label{fig:replace_edge}
  \end{figure}
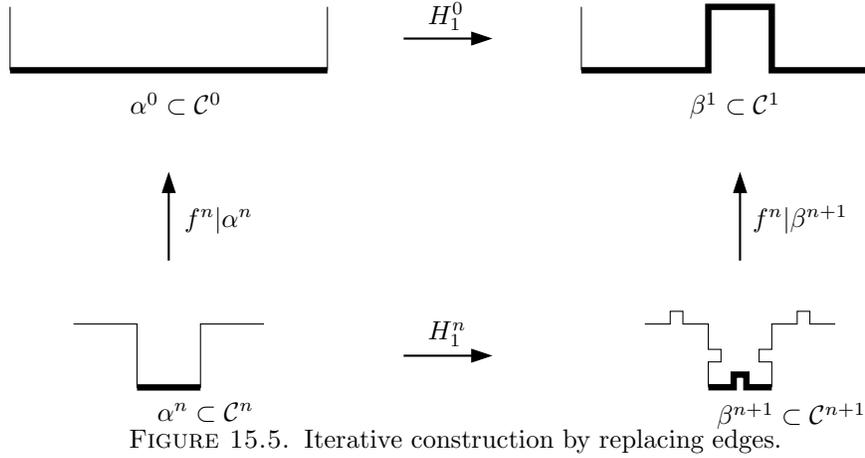
}

One can  look at the arc replacement procedure $\alpha^n \to \beta^{n+1}$ from yet another point of view. Since $\alpha^n$ is an $n$-edge for $(f,\CC^0)$, the map $f^n|\alpha^n$ is a homeomorphism of $\alpha^n$ 
onto the $0$-edge 
 $\alpha^0\coloneqq f^n(\alpha^n)\sub \CC^0$ for $(f, \CC^0)$
 (Proposition~\ref{prop:celldecomp}~\ref{item:fkcellular}). The endpoints of $\alpha^0$ lie in $\post(f)$.  Then  
 $\beta^1\coloneqq H^0_1(\alpha^0)$ is the unique subarc of $ \CC^1$ that has  
  the same endpoints as $\alpha^0$, but contains no other points in 
  $\post(f)$ (here it is important that $\#(\CC^1\cap\post(f))=\#\post(f)\ge 3)$.  Since 
  $H^0_1\circ f^n|\alpha^n$ is a homeomorphism of $\alpha^n$ onto 
$\beta^1$, $f^n\circ H^n_1=H^0_1\circ f^n$, and   $\beta^{n+1}=H^n_1(\alpha^n)$,  the map $f^n|\beta^{n+1}$ is  a homeomorphism of $\beta^{n+1}$ onto 
  $\beta^1$. 
Often, this information (together with the fact that $\alpha^n$ and $\beta^{n+1}$ share endpoints) is enough to determine 
  $\beta^{n+1}$ uniquely. 
We illustrate this procedure in Figure
  \ref{fig:replace_edge}. Here the map $f$ (as well as the curves
  $\CC^0, \CC^1,\dots$ and the isotopies $H^0,H^1,\dots$) are as in
  Example \ref{ex:Cit};  see also Figure \ref{fig:Cit}.

 For example, suppose that  $\beta^1$ lies in a {\em single}  $0$-tile
 $X^0$ for $(f, \CC^0)$, i.e., in one of the Jordan regions bounded by
 $\CC^0$. This is not always true, but in Example~\ref{ex:Cit} as well
 as the  Examples~\ref{ex:Cinv_notcexp} and \ref{ex:rect}  
  discussed below this is the case.   Then there exists a unique 
  $n$-tile  $X^n$ for $(f, \CC^0)$ with $\alpha^n\sub \partial X^n$ and $f^n(X^n)=X^0$; if we assign colors to tiles  for $(f, \CC^0)$  as  in Lemma~\ref{lem:colortiles}, 
then $X^n$ is the unique $n$-tile for $(f,\CC^0)$ that contains 
  $\alpha^n$ in its boundary and has the same color as $X^0$.  
  
   Consider the arc  $\widetilde \beta^{n+1}\coloneqq (f^n|X^n)^{-1}(\beta^1)\sub X^n$. Then 
 $\widetilde \beta^{n+1}$  has the same endpoints as 
 $(f^n|X^n)^{-1}(\alpha^0)=\alpha^n$ and is contained in $f^{-n}(\CC^1)$.
 Moreover, $\widetilde \beta^{n+1}$ is isotopic to $\alpha^n$ rel.\ $f^{-n}(\post(f))$; this easily follows from Lemma~\ref{twoarcs}, since our assumptions imply that one can find a suitable   simply connected region  $\Om\sub S^2$ that contains $\widetilde \beta^{n+1}$ and $\alpha^n$ and no point in $f^{-n}(\post(f))$ except the  endpoints of $\widetilde \beta^{n+1}$ and $\alpha^n$. 
By what we have seen above,  we conclude $\beta^{n+1}=\widetilde \beta^{n+1}$, and so 
 \begin{equation}\label{eq:arcreplacement}
 \beta^{n+1}=(f^n|X^n)^{-1}(\beta^1). 
  \end{equation} 
 
  In the special case under consideration, this leads to a very convenient edge replacement procedure that can be summarized as follows: 
  Suppose  the arc $\beta^1\sub \CC^1$ corresponding to
  $\alpha^0=f^n(\alpha^n)\sub \CC^0$ lies in a single  $0$-tile
  $X^0$, and let $X^n$ be the $n$-tile that contains $\alpha^n$
  in its boundary and has the same color as $X^0$ (so that   $f^n(X^n)=X^0)$. Then $\alpha^n$  is replaced with  the arc $\beta^{n+1}$ in $X^n$ that corresponds to 
  $\beta^1\sub X^0$ under the homeomorphism $f^n|X^n$ of $X^n$ onto $X^0$.    \end{rem}

The next example illustrates what happens if the map 
$\widehat{f}$ in Proposition~\ref{prop:invCit} is not combinatorially expanding. 

\begin{ex}
  \label{ex:Cinv_notcexp}
  Let $g\colon \CDach \to \CDach$ be the Latt\`{e}s map obtained
  according to
  Theorem~\ref{thm:Lattesstruc}~\ref{item:Lattessruciii}  as a
  quotient of the map 
  \begin{equation*}
    A\: \C\ra \C, \quad u\mapsto A(u)\coloneqq 3u,
  \end{equation*}
  by a crystallographic group of type $(2222)$ as in
  \eqref{eq:specisom}.
  The map $g$ was already considered in
  Example~\ref{ex:exp_notcexp} and is represented by  
the bottom part of   Figure~\ref{fig:exp_notcexp}.
 We can identify $\CDach$ with a
  pillow that is obtained by gluing two squares together so that   the set 
  $\post(g)$ consists of the four vertices of the
  pillow.

Let $\CC^0$ be
  the equator of the pillow. The curve
  $\CC^1\subset g^{-1}(\CC^0)$ is drawn with a thick line on the
  top left in  Figure \ref{fig:Cinv_notexp}. Clearly, 
 there is  an
  isotopy $H^0$ rel.\ $\post(g)$  that deforms $\CC^0$ to $\CC^1$. Note that
  $\widehat{g}=H^0_1\circ g$ is not combinatorially expanding for
  $\CC^1$ (see Figure \ref{fig:exp_notcexp}).  Starting with the
  data $\CC^0$, $\CC^1$, $H^0$, we can inductively define Jordan
  curves $\CC^n$ as described before.
  

Based on the discussion in Remark~\ref{rem:iterate_given_Ct}, one can obtain $\CC^{n+1}$ from $\CC^n$ by an edge  replacement procedure.  
It is  determined  by  how a $0$-edge is  replaced with an  arc consisting of $1$-edges in the transition from $\CC^0$ to $\CC^1$.
  In particular, each of the two $0$-edges drawn horizontally in Figure~\ref{fig:Cinv_notexp} is replaced with  itself. 
Since every
  horizontal $n$-edge for $(g, \CC^0)$ is mapped by $g^n$ to a horizontal $0$-edge,
  it is also replaced with  itself in the transition  of 
  $\CC^{n}$
  to  $\CC^{n+1}$; so  if in one step we obtain a horizontal
  edge, then it remains unchanged in subsequent steps.

It follows that $\CC^n\to \widetilde{\CC}$ as $n\to \infty$ in the sense
  of Hausdorff convergence, where the set $\widetilde{\CC}$ is as 
  indicated on the right in  Figure \ref{fig:Cinv_notexp}. The set  $\widetilde{\CC}$ is not a Jordan curve and
  $\CDach\setminus \widetilde{\CC}$ has three components.  For  more  general maps the ``self-intersections'' of such a  limit set
  $\widetilde{\CC}$ can of course be more complicated. 
\end{ex}

\ifthenelse{\boolean{nofigures}}{}{
\begin{figure}
  \centering
  \begin{overpic}
    [width=12cm, 
    tics=20]{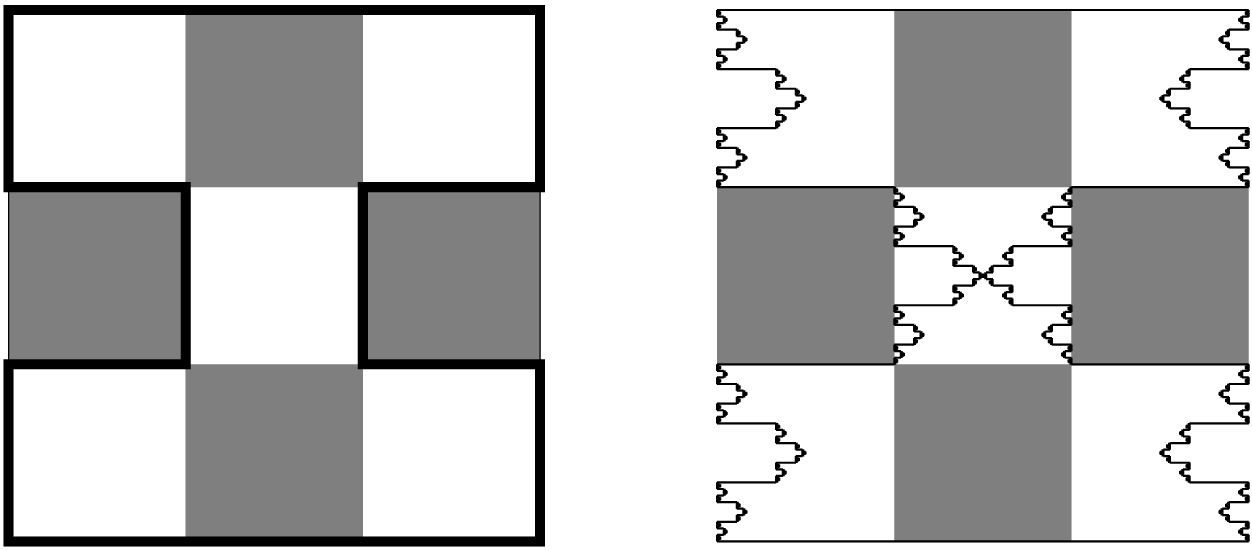}
    \put(7,30){$\CC^1$}
    \put(-3.5,20){$\CC^0$}
    \put(54,2){$\widetilde{\CC}$}
  \end{overpic}
  \caption[Example where $\widetilde{\CC}$ is not a Jordan curve.] 
  {Since $\widehat{g}$ is not combinatorially expanding,
    $\widetilde{\CC}$ is not a Jordan curve.} 
  \label{fig:Cinv_notexp}
\end{figure}
}

\ifthenelse{\boolean{nofigures}}{}{
 \begin{figure}
   \centering
   \begin{overpic} 
     [width=12cm, 
     tics=20]{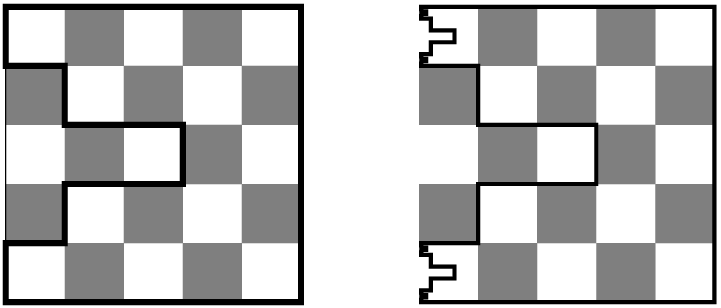}
     \put(-3,3){$\CC^1$}
     \put(54,3){$\widetilde{\CC}$}
   \end{overpic}
   \caption{A non-trivial rectifiable invariant Jordan curve.}
   \label{fig:not_rect_invC}
 \end{figure}
}
  
We conclude  this section with one more example.
It shows a non-trivial invariant curve that is rectifiable.

\begin{ex}
  \label{ex:rect} Let $f$ be the map from Example
  \ref{ex:Cit}, i.e., the  Latt\`{e}s map  obtained as in
  (\ref{eq:Lattes}), where we choose     $A\colon \C\to \C$, $u
  \mapsto A(u)\coloneqq 5u$. 
  The curve $\CC=\CC^0$ is the equator of the pillow as before, and we consider cells for $(f,\CC)$. 
   On the pillow the map $f$
  sends the lower left $1$-tile to the white $0$-tile by the map
  $u\mapsto 5u$ and extends to other $1$-tiles by reflection.
 
   The  curve $\CC^1$ (which is isotopic to $\CC^0$ rel.\ $\post(f)$ by an isotopy
  $H^0$) is the thick curve indicated on the left in  Figure
  \ref{fig:not_rect_invC}. Note that no $1$-tile for $(f,\CC^0)$
joins opposite sides of $\CC^1$. Thus the sequence of curves
  $\{\CC^n\}$, defined  
  as before, Hausdorff converges to an $f$-invariant Jordan curve
  $\widetilde{\CC}$ by Proposition \ref{prop:nscombexp} and
  Proposition \ref{prop:invCit}. 

    Note that the three $0$-edges on the top,
  bottom, and right side of the pillow are deformed by $H^0$ to
  themselves. 
 This means that each  $n$-edge (for $(f,\CC^0)$) in $\CC^n$ that is  sent to one of these $0$-edges by $f^n$ remains unchanged in the passage from $\CC^n$ to $\CC^{n+1}$. 

  The resulting $f$-invariant Jordan curve $\widetilde{\CC}$ is shown
  on the right. It is not hard to see that  $\widetilde{\CC}$ is a  rectifiable curve on the pillow. Indeed, if as before we identify the top square of the pillow with $[0,1/2]^2$, then the $n$-edges have length 
$5^{-n}/2$. The curve  $\CC^n$ contains $2^n$ ``alive'' $n$-edges that will not remain unchanged in subsequent steps. In $\CC^{n+1}$ each of them is replaced with  $11$ edges of 
level $n+1$. A simple computation gives
$$\length(\CC^{n+1})= \length(\CC^{n})+\tfrac 35 (2/5)^{n}, $$  
which implies that indeed $\length(\widetilde \CC)<\infty$. 
\end{ex}

\section{Invariant curves are  quasicircles}
\label{sec:cc-quasicircle}
  Recall from Section~\ref{sec:QCgeom} that 
 a metric circle $(S,d)$ is called a  quasicircle
 if it
is quasisymmetrically equivalent to  the unit circle  in
$\R^2$ (equipped with the Euclidean  metric).  
This is the case if and only if 
$(S,d)$ is doubling and of bounded turning (see
Theorem~\ref{thm:TV}).  
  We will now verify that this is true for an invariant curve 
as in  Theorem~\ref{thm:Cquasicircle}.

\begin{proof}[Proof of Theorem \ref{thm:Cquasicircle}] Suppose $\CC$ is an $f$-invariant Jordan cur\-ve as in the statement, and  
let $\varrho$ be a visual metric on $S^2$ with  expansion factor  $\Lambda>1$. Metric notions will be for  this metric in the following. 

  In the ensuing proof, we will consider edges for $(f,\CC)$. 
Since $\CC$ is $f$-invariant, edges are subdivided by edges of higher levels (see Proposition~\ref{prop:invmarkov}~\ref{item:invmarkov4}). The Jordan curve  $\CC$ is the  union of all $0$-edges;  so this implies that $\CC$ is a union of $n$-edges for all 
$n\in \N_0$.  
If $n,k\in \N_0$ and $\widetilde e$ is an arbitrary $(n+k)$-edge with 
$\widetilde e\sub \CC$, then there exists a unique $n$-edge $e'$ with $\widetilde e\sub e'\sub \CC$.  


If $e'$ is an $n$-edge, then the number of $(n+k)$-edges
$\widetilde e$ contained in $e'$ is bounded by
$\#\post(f)\deg(f)^k$. Indeed, the map $f^n|e'$ is injective; so
the images of these $(n+k)$-edges $\widetilde e$ under the map
$f^n$ are distinct $k$-edges, and the number of $k$-edges is
equal to $\#\post(f)\deg(f)^k$ (see
Proposition~\ref{prop:celldecomp}~\ref{item:noVEX}).

After these preliminaries, we are ready to show that $\CC$ equipped with (the restriction of) $\varrho$ is a quasicircle. We first establish that 
$\CC$ is doubling.  Note that in contrast $(S^2, \varrho)$ is not doubling in general (see Theorem~\ref{thm:S2vsf}~\ref{item:S2doubling}).

Let  $x\in \CC$, and $0<r\le 2\diam(\CC)$. In order to show that $\CC$  is doubling, it suffices to cover $B(x,r)\cap \CC$ by a controlled number
of sets of diameter $<r/4$.  

It follows from Proposition~\ref{lem:expoexp} that we can find
$n\in \N_0$ depending on $r$, as well as constants $C(\asymp)>0$ and $k_0\in \N_0$  independent of $x$ and $r$  with
the following properties: 
\begin{enumerate}

\item
  \label{item:C_inv_doubling1}
  $r\asymp \Lambda^{-n}$.

\item
  \label{item:C_inv_doubling2}
  $\diam(e)< r/4$, whenever $e$ is an $(n+k_0)$-edge.

\item
  \label{item:C_inv_doubling3}
  $\dist(e,e')\ge r$, whenever  $n-k_0\ge 0$ and $e,e'$ are
  disjoint $(n-k_0)$-edges.   
\end{enumerate} 

Let $E$ be the set of all $(n+k_0)$-edges contained in $\CC$ that meet $B(x,r)$. Then the collection $E$ forms a cover of  $\CC\cap B(x,r)$ and consists of sets of diameter $<r/4$ by \ref{item:C_inv_doubling2}. 
Hence it suffices to find a uniform upper bound for  $\#E$. 
If $n< k_0$, then $\#E\le \#\post(f)\deg(f)^{2k_0}. $

Otherwise, $n-k_0\ge 0$. Then we can find an $(n-k_0)$-edge $e\sub \CC$ with $x\in e$. Let $\widetilde e$ be an arbitrary $(n+k_0)$-edge in $E$. Then  we can find  an $(n-k_0)$-edge $e'\sub \CC$ that contains $\widetilde e$. 

There exists a point $y\in \widetilde e\cap B(x,r)$.
Hence $\dist(e,e')\le \varrho(x,y)<r$. This implies $e\cap e'\ne \emptyset$ by \ref{item:C_inv_doubling3}.
So whatever $\widetilde e\in E$ is, the corresponding $(n-k_0)$-edge $e'\sub \CC$ meets the fixed $(n-k_0)$-edge $e$. This leaves at most three possibilities for $e'$, namely $e$, and the two ``neighbors`` of $e$ on $\CC$. So there are three or less
$(n-k_0)$-edges that contain all the edges
 in $E$.  Since each $(n-k_0)$-edge contains at most 
$\#\post(f)\deg(f)^{2k_0}$ edges of level $(n+k_0)$, it follows that 
$\#E\le 3\#\post(f)\deg(f)^{2k_0}$. In both  cases, we get an upper 
 bound for $\#E$ as desired.

  It remains to show that $\CC$ is of bounded turning. Let  $x,y\in
  \CC$ with $x\ne y$ be arbitrary. 
We want to 
establish the inequality $\diam(\ga)\lesssim \varrho(x,y)$ 
with a uniform constant $C(\lesssim)$ for one of the two subarcs  
  $\ga$ of $\CC$ with  endpoints $x$ and $y$.  
  For this let $n_0\ge 0$ be the smallest integer for which there exist $n_0$-edges $e_x\sub \CC$ and $e_y\sub \CC$ with $x\in e_x$,  $y\in e_y$,  
  and $e_x\cap e_y=\emptyset$.  Note that $n_0$ is well-defined, because $f$ is expanding and so the diameter of $n$-edges approaches $0$ uniformly as $n\to \infty$. 
  
  Then by  Proposition~\ref{lem:expoexp}~\ref{item:expoex1}, 
  $$\varrho(x,y)\gtrsim \Lambda^{-n_0}.$$ If $n_0=0$, then 
  $$\diam(\CC)\lesssim \varrho(x,y)$$ and there is nothing to prove.
  If $n_0\ge 1$, we can find $(n_0-1)$-edges $e'_x\sub \CC$
  and $e'_y\sub \CC$ with $x\in e'_x$, $y\in e'_y$, and $e'_x\cap e'_y\ne \emptyset$. Then $e'_x\cup e'_y$ must contain 
  one of the subarcs $\ga$ of $\CC$ with endpoints $x$ and $y$.
  Hence 
  \begin{equation*}
    \diam (\ga)
    \le 
    \diam(e'_x)+\diam(e'_y)
    \lesssim
    \Lambda^{-n_0}
    \lesssim 
    \varrho(x,y).  
  \end{equation*}
  Since the implicit
  multiplicative constants in the previous inequalities do not
  depend on $x$ and $y$, we get a bound as desired.  
\end{proof}

Recall that a metric is  visual for $f$ if and only only if it is
visual for any  iterate of $f$ (see
Proposition~\ref{prop:visualsummary}~\ref{item:vsfF}). Hence we
may apply Theorem~\ref{thm:Cquasicircle} to any 
Jordan curve $\CC\subset S^2$ with $\post(f)\subset \CC$ that is invariant
for an iterate of $f$.
In particular, the invariant Jordan curve in
Theorem~\ref{thm:main} is a 
quasicircle if equipped with a visual metric for $f$. 

A family of quasisymmetries (possibly defined on different spaces) is 
called  {\em uniformly 
quasisymmetric}\index{quasisymmetry!uniform}\index{uniformly!quasisymmetric}
if there exists a
homeomorphism $\eta\:[0,\infty)\ra [0,\infty)$ such that each map in
the family is an $\eta$-quasi\-symmetry.  Obviously, each finite
family of quasisymmetries is uniformly quasisymmetric.  If $h$ is an
$\eta$-quasisymmetry, then $h^{-1}$ is  an
$\widetilde\eta$-quasisymmetry, where $\widetilde \eta$ only depends
on $\eta$; actually, one can take $\widetilde \eta\: [0,\infty)\ra
[0,\infty)$ defined by $\widetilde \eta(0)=0$ and  $\widetilde
\eta(t)=1/\eta^{-1}(1/t)$ for $t>0$.  
This implies that if a family of maps is uniformly quasisymmetric, then the family of inverse maps is also uniformly quasisymmetric.

If $X,Y,Z$ are metric spaces, $h_1\: X\ra Y$ is $\eta_1$-quasisymmetric, and $h_2\: Y\ra Z$ is $\eta_2$-quasisymmetric, then 
$h_2\circ h_1$ is $\eta$-quasisymmetric, where $\eta=\eta_2\circ \eta_1$. Hence the family of all compatible compositions of maps in two 
uniformly quasisymmetric families  is again uniformly quasisymmetric. 

Recall (see Section~\ref{sec:QCgeom}) that an arc $\alpha$ equipped
with some metric $d$ is called a {\em quasiarc}
if there exists a quasisymmetry of the  unit interval $[0,1]$  onto $(\alpha, d)$. This is true if and only if $(\alpha, d)$ is doubling and there exists a constant $K\ge 1$ such that 
$\diam_d(\gamma)\le K d(x,y)$, whenever $x,y\in \alpha$   and $\gamma$ is the  subarc of $\alpha$ with endpoints  $x$ and $y$ (see Theorem~\ref{thm:TV}).

A family of arcs is said to consist of {\em uniform
  quasiarcs}\index{quasiarc!uniform}\index{uniform!quasiarcs}
 if there 
exists a homeomorphism $\eta\:[0,\infty)\ra [0,\infty)$ such that for
each arc $\alpha$ in the family there exists an $\eta$-quasisymmetry
$h\:[0,1]\ra \alpha$. Similarly, a family of quasicircles  is said to
consist of 
{\em uniform
 quasicircles}\index{quasicircle!uniform}\index{uniform!quasicircles} 
if there 
exists a homeomorphism $\eta\:[0,\infty)\ra [0,\infty)$ such that for each
quasicircle  $S$ in the family there exists an $\eta$-quasisymmetry
$h\:\partial \D \ra S$.  A family of quasicircles consists of uniform
quasicircles if and only if the geometric conditions characterizing
quasicircles, i.e.,  
the doubling condition and the bounded turning  condition, 
hold with uniform parameters. 
A similar statement is true for families of
quasiarcs (see \cite{TV}). 

We want to  show that if the assumptions are as in
Theorem~\ref{thm:Cquasicircle}, then all boundaries of  tiles for
$(f,\CC)$ are quasicircles and all edges for $(f,\CC)$ are
quasiarcs.  Actually, the family of all  boundaries of  tiles
consists of uniform quasicircles and the family of all edges
consists of uniform quasiarcs. One way to establish  this is to
repeat the proof of Theorem~\ref{thm:Cquasicircle} and show that the geometric conditions characterizing quasiarcs and quasicircles are true for the edges and boundaries of tiles with uniform constants.  We choose a different approach that is based on the following lemma which is of independent 
interest.

\begin{lemma}
  \label{lem:unifqs} 
  \index{uniformly!quasisymmetric}
  \index{quasisymmetry!uniform}
  \index{visual metric}\index{metric!visual}\index{r@$\varrho$}
  Let $f\:S^2\ra S^2$ be an expanding Thurston map,  and $\CC\sub S^2$ be an $f$-invariant Jordan curve with 
$\post(f)\sub \CC$. Suppose that $S^2$ is equipped with a visual metric $\varrho$ for $f$ with expansion factor $\Lambda>1$,
and   
denote by $\X^n$ for $n\in \N_0$ the set of $n$-tiles for $(f,\CC)$. Then there exists a constant $C\ge 1$ with the following property:

If $k,n\in \N_0$, $X^{n+k}\in \X^{n+k}$, and $x,y\in X^{n+k}$, then 
\begin{equation}\label{eq:unifsim}
 \frac1C \varrho(x,y)\le \frac{\varrho(f^n(x), f^n(y))}{\Lambda^{n}}\le C\varrho(x,y).
 \end{equation} 

In particular, the family 
$$\mathcal{F}=\{ f^n|X^{n+k}: k,n\in \N_0,\,   X^{n+k}\in \X^{n+k}\}$$ 
is uniformly quasisymmetric. 
\end{lemma}
The distortion estimate \eqref{eq:unifsim} is closely related to the concept of a 
{\em conformal elevator}\index{conformal elevator}  
as introduced by Ha\"\i ssinsky and Pilgrim \cite[Theorem~2.2]{HP}.  See also \eqref{simmetric} in Theorem~\ref{thm:visexpfactors1} for a related statement.

\begin{proof}  In the following, all cells will be for $(f,\CC)$. Let $m=m_{f,\CC}$ be as in Definition~\ref{def:mxy}. 
We know by Definition~\ref{def:visual} and by  Lem\-ma~\ref{lem:mprops}~\ref{item:mprops3}
that $\varrho(x,y)\asymp \Lambda^{-m(x,y)}$, whenever $x,y\in S^2$.  
If $n\in \N_0$, then Lemma~\ref{lem:mprops}~\ref{item:mprops2} implies that 
$$m(f^n(x), f^n(y))\ge m(x,y)-n, $$ and so
$$ \varrho(f^n(x), f^n(y))\lesssim \Lambda^n \varrho(x,y). $$
Here the implicit multiplicative constant is independent of $x$, $y$, and $n$.

To obtain an inequality in the other direction,  let  $x,y\in X^{n+k}\in 
\X^{n+k}$, where $n,k\in \N_0$. We may assume that  $x\ne y$. Then by definition of 
$m(x,y)$ we have $n+k\le m(x,y)<\infty$. Let   $l\coloneqq m(x,y)+1\in \N$. 
Since $l>n+k$,  the $(n+k)$-tile $X^{n+k}$ is subdivided by tiles of level 
$l$ (Proposition~\ref{prop:invmarkov}~\ref{item:invmarkov3}). Hence there exist $l$-tiles 
$X,Y\sub X^{n+k}$ with $x\in X$ and $y\in Y$. 
Then $X\cap Y=\emptyset$ by definition of $m(x,y)$. Let $X'\coloneqq f^n(X)$ and 
$Y'\coloneqq f^n(Y)$. Then by Proposition~\ref{prop:celldecomp}~\ref{item:fkcellular} the sets $X'$ and $Y'$ are $(l-n)$-tiles. Since $f^n|X^{n+k}$ is injective, these tiles are disjoint, and we have $f^n(x)\in X'$ and $f^n(y)\in Y'$. 
So from Proposition~\ref{lem:expoexp}~\ref{item:expoex1} we conclude that 
$$ \varrho(f^n(x), f^n(y))\ge \dist_\varrho(X',Y')\gtrsim \Lambda^{-(l-n)}\asymp \Lambda^{n}\Lambda^{-m(x,y)}\asymp \Lambda^{n} \varrho(x,y). $$
Here  the implicit multiplicative constants are again independent of $x$, $y$, and $n$. 
The  other desired inequality follows. 
 
Inequality~\eqref{eq:unifsim} immediately implies that the family $\mathcal{F}$  is uniformly quasisymmetric. To see this, let $k,n\in \N_0$ and $X^{n+k}\in \X^{n+k}$. Then $f^n|X^{n+k}$  is a homeomorphism onto its image (see Proposition~\ref{prop:celldecomp}~\ref{item:fkcellular}). Moreover, if $u,v,w\in X^{n+k}$, $u\ne w$, then by \eqref{eq:unifsim} we have 
$$ \frac{\varrho(f^n(u), f^n(v))}{\varrho(f^n(u), f^n(w))}\le C^2 \frac{\varrho(u,v)}{\varrho(u,w)}.$$ 
Hence $f^n|X^{n+k}$ is $\eta$-quasisymmetric, where $\eta(t)=C^2t$ for $t\ge 0$. Since $\eta$ is independent of the chosen map, 
the family $\mathcal{F}$ is uniformly quasisymmetric. 
 \end{proof}

\begin{prop} \label{prop:arc} Let $f\:S^2\ra S^2$ be an expanding Thurston map,  and $\CC\sub S^2$ be an $f$-invariant Jordan curve with 
$\post(f)\sub \CC$. Suppose that $S^2$ is equipped with a visual metric for $f$, and for $n\in \N_0$ denote by $\X^n$ the set of $n$-tiles and by  $\E^n$ the set of $n$-edges 
 for $(f,\CC)$. 

Then the family   $ \{\partial X:  n\in \N_0, \,  X\in \X^n\}$ consists of 
uniform  quasicircles and the family 
$\{e: n\in \N_0, \,  e\in \E^n\}$  of uniform quasiarcs.  
\end{prop}
In particular,  edges for $(f,\CC)$ are quasiarcs and the boundaries   of all tiles are quasicircles. 

\begin{proof} By Theorem~\ref{thm:Cquasicircle} there exists a quasisymmetry $h\: \partial \D \ra \CC$. Let $X$ be an arbitrary tile for 
$(f,\CC)$, say an $n$-tile, where $n\in \N_0$. Then 
$f^n|X$ is a homeomorphism of $X$ onto the $0$-tile $f^n(X)$ (Proposition~\ref{prop:celldecomp}~\ref{item:fkcellular}), and so
$$ f^n(\partial X)=\partial f^n(X)=\CC. $$
By Lemma~\ref{lem:unifqs} the map $f^n|X$,  and hence also the map 
$(f^n|X)^{-1}$,   is a quasisymmetry. It follows that 
$(f^n|X)^{-1}\circ h$ is a quasisymmetric map from $\partial \D$ onto 
$\partial X$. Hence $\partial X$ is a quasicircle. 
Actually, the family of these quasicircles $\partial X$ is uniform, since the family of all  relevant  maps $(f^n|X)^{-1}\circ h$ is uniformly quasisymmetric as follows from Lemma~\ref{lem:unifqs}.

The proof that the family $\{e: n\in \N_0, \,  e\in \E^n\}$ consists  of uniform quasiarcs runs along the same lines. First note that each  $0$-edge is a subarc of $\CC$, and hence corresponds to a subarc  of $\partial \D$ under the quasisymmetry $h$. Since this subarc can be 
mapped to the unit interval $[0,1]$ by a bi-Lipschitz homeomorphism, 
each $0$-edge is quasisymmetrically equivalent to $[0,1]$ and hence a quasiarc. 

Now let $e$ be  an arbitrary edge for $(f,\CC)$, say an $n$-edge, where $n\in \N_0$. Then $f^n|e$ is a homeomorphism of $e$ onto the $0$-edge $f^n(e)$ (Proposition~\ref{prop:celldecomp}~\ref{item:fkcellular}). Moreover, there exists an $n$-tile $X$ with $e\sub X$. Then $f^n|e$ is the restriction of the map $f^n|X$ to $e$, and it follows from Lemma~\ref{lem:unifqs} that $f^n|e$ is a quasisymmetry. Hence $e$ is quasisymmetrically equivalent to a $0$-edge and hence a quasiarc. 

 Lemma~\ref{lem:unifqs} actually implies that the family consisting of all 
 maps $f^n|e$ with $n\in \N_0$  and $e\in \E^n$ is uniformly quasisymmetric. So each edge is quasisymmetrically equivalent to a $0$-edge by a quasisymmetry in a uniformly quasisymmetric family. Since there are only finitely many $0$-edges, this implies that the  family of all edges for $(f,\CC)$  consists  of uniform quasiarcs. 
\end{proof}  

A {\em quasidisk}\index{quasidisk}
is a closed topological disk (i.e., a topological cell of dimension $2$) that is quasisymmetrically equivalent to the closed unit disk $\overline \D$.
A family of closed topological disks is said to consist of 
{\em uniform quasidisks}\index{quasidisk!uniform} 
if each disk $X$ in the family can be mapped to $\overline \D$ by an $\eta$-quasisymmetry, where $\eta$ is independent of $X$. It is a natural question whether the family $\{X: n\in \N_0,\, X\in  {\bf X}^n\}$ of tiles obtained from an invariant curve as in the previous theorem actually consists of uniform quasidisks. This is true if and only if the expanding 
Thurston map $f$ is topologically conjugate to a rational map without periodic critical points. One direction easily follows from Theorem~\ref{thm:main01}~\ref{item:frat_Markov} and Corollary~\ref{cor:visualqsmetric} proved later. 

For the other direction suppose that $f$ and $\CC$ are as in Proposition~\ref{prop:arc} and that the  two $0$-tiles equipped with a visual metric $\varrho$ are quasidisks. Then one can show  that 
$(S^2,\varrho)$  is a quasisphere (this requires the solution of a so-called 
{\em welding problem}). So by Theorem~\ref{thm:S2vsf}~\ref{item:S2qsphere}
the map $f$ is topologically conjugate to a rational map without periodic critical points. We skip the details for this implication as we will not use the result.

\ifthenelse{\boolean{singlechapter}}{

%


\chapter{The combinatorial expansion factor}
\label{cha:combexpfac}

Suppose $f\: S^2 \ra S^2$ is a Thurston map with $\#\post(f)\ge
3$, 
and $\CC\sub S^2$ is a Jordan curve with $\post(f)\sub \CC$. In
Section~\ref{sec:opp} we introduced the quantity $D_n(f,\CC)$ as
the minimal number of $n$-tiles for $(f,\CC)$ required to form a
connected set that joins opposite sides of $\CC$ (see
Definition~\ref{def:connectop} and \eqref{def:dk}).  In this
chapter we study the asymptotic behavior of $D_n(f,\CC)$ as
$n\to \infty$. We will see that for an expanding Thurston map,
$D_n(f,\CC)$ grows at an exponential rate independent of $\CC$.

\index{d0 n@$D_n$}
\begin{prop} 
  \label{prop:exp} 
  Suppose $f\:S^2\ra S^2$ is an expanding Thurston map, and
  $\CC\sub S^2$ is a Jordan curve with $\post(f)\sub \CC$.  
  Then
  the limit
  \begin{equation*}
    \Lambda_0(f)
    \coloneqq 
    \lim_{n\to \infty} D_n(f,\CC)^{1/n}
  \end{equation*}
  exists. Moreover, this limit is independent of $\CC$ and we
  have $1<\Lambda_0(f)<\infty$.
\end{prop} 

We call $\Lambda_0(f)$
 the {\em combinatorial expansion factor}\index{combinatorial expansion factor|textbf}\index{L0@$\Lambda_0$|textbf}  
of $f$.   Later  we will see that $\Lambda_0(f)\le \deg(f)^{1/2}$
(Proposition~\ref{prop:macgrdn}). 

The combinatorial expansion factor is 
well-behaved under taking iterates  and invariant under topological conjugacy.

\begin{prop} \label{prop:expfacinv}
Let $f\: S^2\ra S^2$ be an expanding Thurston map. Then the following statements are true: 

\begin{enumerate} 

 \item 
   \label{item:comexpiterates1} 
   \index{Thurston map!iterate of}
  \index{iterate of Thurston map}
  \index{F f@$F=f^n$}
   $\Lambda_0(f^n)=\Lambda_0(f)^n$ for  $n\in \N$. 
 
  \item \label{item:comexpiterates2}
Suppose  
$g\:\widehat S^2\ra \widehat S^2$ is an   expanding Thurs\-ton map that is topologically conjugate to $f$. Then 
$\Lambda_0(g)=\Lambda_0(f)$. 
\end{enumerate} 
\end{prop}

The main result of this chapter relates the combinatorial
expansion factor to expansion factors of visual metrics.

\index{expansion factor}
\begin{theorem}[Visual metrics and their expansion factors] 
  \label{thm:visexpfactors1} 
  Let $f\:S^2\ra S^2$ be an expanding Thurston map, and
  $\Lambda_0(f)\in (1,\infty)$ be its combinatorial expansion
  factor. Then the following statements are true:
  \begin{enumerate}
    
  \item
   \label{item:visexpfactors1}
   If   $\Lambda$ is the expansion factor of a  visual metric 
 for $f$, then   $1<\Lambda\le \Lambda_0(f)$.
 
 \item
   \label{item:visexpfactors2}
   Conversely, if $1<\Lambda<\Lambda_0(f)$, then there exists a visual metric $\varrho$ for $f$ with expansion factor $\Lambda$. Moreover, the visual metric $\varrho$ can be chosen to have the following additional property:
 
\index{metric!visual}
\index{visual metric}
 For every $x\in S^2$ there exists a neighborhood $U_x$ of $x$ such that 
 \begin{equation} \label{simmetric}
 \varrho(f(x), f(y))=\Lambda  \varrho(x,y)\text{ for all } y\in U_x.
\end{equation} 
\end{enumerate}
\end{theorem} 

This statement shows  that if  $\Lambda$ is the expansion factor of a visual metric for $f$, then 
 $1<\Lambda\le \Lambda_0(f)$, but conversely, the existence of a visual metric with expansion 
 factor $\Lambda$ is only guaranteed for  $1<\Lambda< \Lambda_0(f)$.  This statement is optimal, since  a visual metric with expansion factor $\Lambda=\Lambda_0(f)$ need  not exist in general. We will discuss an example at the end of this chapter
(see Example~\ref{ex:notattained}). 
However, in the proof of
Theorem~\ref{thm:visexpfactors1}~\ref{item:visexpfactors2}
 we will establish  an existence statement that is somewhat  stronger: if $\CC\sub S^2$ is an $f$-invariant Jordan curve with $\post(f)\sub \CC$ and 
 $1<\Lambda \leq D_1(f,\CC)$, then there exists a visual metric $\varrho$ for $f$ with expansion factor $\Lambda$ (see \eqref{extraonL}).

We now proceed to supply the proofs. We require   some preparation and  start 
with some lemmas.

\begin{lemma}\label{lem:flowpow} Let $n\in \N_0$, $f\:S^2\ra S^2$ be a Thurston map with $\#\post(f)\ge 3$, and $\CC\sub S^2$ be a Jordan curve 
with $\post(f)\sub \CC$.  If there exists a connected set $K\sub S^2$ that joins 
opposite sides of $\CC$ and that can be covered by $M\in \N$ $n$-flowers
for $(f,\CC)$, then $D_n(f,\CC)\le 4M$.   \end{lemma}

\begin{proof} We first assume  that $\#\post(f)=3$. 
Let $K$ be  as in the statement. 
By picking a point from the intersection of $K$ with each of the three $0$-edges, we can find a set  $\{x,y,z\}\sub K$ such that 
$\{x,y,z\}$  joins  opposite sides of $\CC$.
Since $K$ is connected and can be covered by  $M$ $n$-flowers, we can find $n$-vertices $v_1, \dots, v_M\in S^2$ such that 
$x\in W^n(v_1)$, $y\in W^n(v_M)$,  and 
$W^n(v_i)\cap  W^n(v_{i+1})\ne \emptyset $ for $i=1, \dots, M-1$.
Then it follows from Lemma~\ref{lem:flowerprop}~\ref{item:flower_prop2} that there exists 
a chain of $n$-tiles $X_1, \dots, X_{2M}$  joining $x$ and $y$
(recall the terminology from Definition~\ref{def:chains}). 
Similarly,  there exists a chain $X_1', \dots, X_{2M}'$ of $n$-tiles joining $x$ and $z$. The union $K'$ of the $n$-tiles in these two chains is a connected set consisting of at most $4M$ $n$-tiles.
It contains the set $\{x,y,z\}$ and hence joins opposite sides of $\CC$. Thus  $D_n(f,\CC)\le 4M$. 

If $\#\post(f)\ge 4$, the proof is similar and easier. In this case we can find a set $\{x,y\}\sub K$ that joins opposite sides of $\CC$. 
By the same argument as before, we get  the bound $D_n(f,\CC)\le 2M$. \end{proof}

\begin{lemma}
  \label{lem:Dnprops} 
  Let $f\:S^2\ra S^2$ be an expanding Thurston map, and
  $\CC, \widetilde\CC\sub S^2$ be Jordan curves with
  $\post(f)\sub \CC, \widetilde{\CC}$. 
  Then
  \begin{equation}
    \label{DDk} 
    D_n(f,\CC)
    \asymp
    D_{n+1}(f,\CC)
  \end{equation}  
  and 
  \begin{equation}
    \label{DDtilde}
    D_n(f,\CC) 
    \asymp 
    D_n(f,\widetilde{\CC}) 
  \end{equation} 
  for all $n\in \N_0$, where $C(\asymp)$ is independent of $n$. 
\end{lemma}

\begin{proof}  Set $D_n=D_n(f,\CC)$ and $\widetilde D_n=D_n(f,\widetilde{\CC})$ for  $n\in \N_0$.   

To show \eqref{DDk},
we fix $n\in \N_0$ and pick a connected set $K$ joining  
opposite sides of $\CC$ that consists of $D_n$ $n$-tiles for
$(f,\CC)$. According to Lemma~\ref{lem:difflevel}~\ref{item:coverflower2} we can cover $K$ by $MD_n$ $(n+1)$-flowers, where $M\in \N$ is independent of $n$.  Hence by Lemma~\ref{lem:flowpow} we have 
$D_{n+1}\le C D_n$, where $C=4M$. An inequality in the opposite direction follows from a similar argument  based on 
Lemma~\ref{lem:difflevel}~\ref{item:coverflower1} and
Lemma~\ref{lem:flowpow}.  

To establish \eqref{DDtilde},  we consider 
 $\tilde \delta_0=\delta_0(f,\widetilde \CC)>0$   defined  as in \eqref{defdelta} for $f$, $\widetilde \CC$, and a base metric $d$ on $S^2$. Since $f$ is expanding, there exists  $n_0\in \N_0$ such that 
$\diam_d(X)<\tilde  \delta_0/2$, whenever $X$ is an
$n_0$-tile for $(f,\CC)$. 

We can find a compact connected set $\widetilde K$ joining opposite sides of $\widetilde{\CC}$ that consists of $\widetilde D_n$ $n$-tiles for $(f,\widetilde{\CC})$.  Then $\diam_d( \widetilde K)\ge \tilde \delta_0$ and so $\widetilde K$ contains two points $x$ and $y$  with $d(x,y)\ge \tilde\delta_0$. There  exist $n_0$-tiles $X$ and $Y$  for $(f,\CC)$ such that $x\in X$ and $y\in Y$. By choice of $n_0$ we have $X\cap Y=\emptyset$, and so  
 $\widetilde K$ joins  $n_0$-tiles for $(f,\CC)$  that are disjoint.  
 Hence $f^{n_0}(\widetilde K)$ joins opposite sides of 
 $\CC$ by Lemma~\ref{lem:maptotop}.  Every $n$-tile
for $(f,\widetilde {\CC})$ can be covered by $M$ $n$-flowers for $(f,\CC)$, where $M$ only depends on $\CC$ and $\widetilde{\CC}$ (Lemma~\ref{lem:tileflower}). 
This and Lemma~\ref{lem:mapflowers}~\ref{item:mapflowers3} imply  that if $n\ge n_0$, then we can cover $f^{n_0} (\widetilde K)$ by $M\widetilde D_n$ $(n-n_0)$-flowers for $(f,\CC)$. 

So  by Lemma~\ref{lem:flowpow} we have 
$$D_{n-n_0}\le 4M  \widetilde D_n,$$
and by the first part of the proof there is a constant
$C_1>0 $ such that
\begin{equation*}
  D_n
  \le 
  C_1^{n_0} D_{n-n_0}
  \le 
  4M  C_1^{n_0}\widetilde D_n.
\end{equation*}
If  $n< n_0$ we get a similar bound from the inequalities  $D_n\le 2\deg(f)^{n_0}$ and $\widetilde D_n\ge 1$.  It follows that  there exists a constant $C$ independent of $n$ such that 
$$ D_n\le C \widetilde {D}_n$$
for all $n\in \N_0$. An  inequality in the opposite direction is
obtained by reversing the roles of $\CC$ and $\widetilde {\CC}$
and using an estimate analogous to \eqref{DDk} for $\widetilde
D_n$.  
\end{proof}

\begin{proof}[Proof of Proposition~\ref{prop:exp}]
A consequence of \eqref{DDtilde} is that if $\CC, \widetilde {\CC}\sub S^2$ are Jordan curves with $\post(f)\sub \CC, \widetilde {\CC}$ and  the sequence
$ \{D_n(f,\CC)^{1/n}\}$ converges as $n\to \infty$, then  $ \{D_n(f,\widetilde \CC)^{1/n}\}$ also converges and has the same limit. 
So if the limit exists, then  it does not depend on $\CC$. 

To show existence, we may impose additional assumptions on $\CC$; namely by Theorem~\ref{thm:main}, we may assume that $\CC$ is invariant for some iterate $F=f^N$ of $f$.  Since $F$ is  also an expanding Thurston map (Lemma~\ref{lem:Thiterates}), it follows from
 Lemma~\ref{lem:Dtoinfty} and
 Lemma~\ref{lem:submultexp} that the limit  
$$\Lambda_0(F,\CC)\coloneqq \lim_{n\to\infty} D_n(F,\CC)^{1/n}$$ exists and that
$\Lambda_0(F,\CC)\in (1,\infty)$.

Since the $n$-tiles for $(F,\CC)$ are precisely the $(nN)$-tiles
for $(f,\CC)$ (see Proposition~\ref{prop:celldecomp}~\ref{item:celdecompiter}), we have $D_{nN}(f,\CC)=D_n(F,\CC)$ for all $n\in \N_0$, and so 
$$D_{nN}(f,\CC)^{1/(nN)}=D_n(F,\CC)^{1/(nN)}\to \Lambda_0(f)\coloneqq \Lambda_0(F,\CC)^{1/N}\in (1, \infty)$$
as $n\to \infty$. 
Combining this with \eqref{DDk}, we conclude that 
$D_n(f,\CC)^{1/n}\to \Lambda_0(f)$ as $n\to \infty$. The statement follows. 
\end{proof}

\begin{proof}[Proof of Proposition~\ref{prop:expfacinv}]
\ref{item:comexpiterates1} If $F=f^n$ is an iterate of $f$, then, as was pointed out in the previous proof, we have
$$D_k(F,\CC)=D_{nk}(f,\CC)$$ whenever $k\in \N_0$ and $\CC$ is a Jordan curve with $\post(f)\sub \CC$. 
This implies
\begin{equation*}
\Lambda_0(f^n)=\Lambda_0(F)=\lim_{k\to \infty} D_k(F,\CC)^{1/k}=\lim_{k\to \infty} D_{nk}(f,\CC)^{1/k}=\Lambda_0(f)^n. 
\end{equation*}

\smallskip
\ref{item:comexpiterates2} By assumption there  exists a homeomorphism $h\: S^2\ra \widehat S^2$ such that $h\circ f=g\circ h$. Pick a Jordan curve $\CC\sub S^2$ with $\post(f)\sub \CC$ and let $\widehat \CC= h(\CC)$. Then  $\widehat \CC$ is a Jordan curve with $\post(g)=h(\post(f))\sub \widehat \CC$, and, as in the  proof of Proposition~\ref{prop:conjisom}, we have
$$\DD^n(g, \widehat \CC)=\{h(c): c\in \DD^n(f,\CC)\}$$ for $n\in \N_0$. This  implies that $D_n(f,\CC)=D_n(g, \widehat \CC)$ for all $n\in \N_0$ and so
$$\Lambda_0(g)=
\lim_{n\to \infty} 
D_n(g,\widehat \CC)^{1/n}=
\lim_{n\to \infty} D_n(f,\CC)^{1/n}=\Lambda_0(f) $$
as desired. 
\end{proof}

We now proceed to prove Theorem~\ref{thm:visexpfactors1},  the  main result of this chapter.   
So let $f\: S^2\ra S^2$ be an expanding Thurston map. We fix a  Jordan
  curve $\CC\sub S^2$  with $\post(f)\sub \CC$, and let
  $D_k=D_k(f,\CC)$ for $k\in \N_0$. 
  In the following, cells will be
  for $(f,\CC)$.   In Proposition~\ref{prop:exp} the
  combinatorial expansion factor
  $\Lambda_0(f)$ was defined, and we proved that $1<
  \Lambda_0(f)<\infty$.  
  
  The proof of the first part of Theorem~\ref{thm:visexpfactors1} is easy.
  
  \begin{proof}[Proof of Theorem~\ref{thm:visexpfactors1}~ \ref{item:visexpfactors1}] 
   Suppose $\varrho$ is a visual metric for $f$ with expansion factor $\Lambda$. Then 
there exists a constant $C\ge 1$ such that 
$$ \diam(X)\le C \Lambda^{-k}$$ 
for all $k$-tiles (Proposition~\ref{lem:expoexp}~\ref{item:expoex2}). Let
$\delta_0=\delta_0(f,\CC)>0$ be  defined as in \eqref{defdelta} for
$f$, $\CC$, and the metric $\varrho$.  

For each  $k\in \N_0$  there exists a connected set $K\sub S^2$ joining
opposite sides of $\CC$ that consists of $D_k$ $k$-tiles. Hence 
\begin{equation}
  \label{eq:DkLged}
  \delta_0
  \le 
  \diam(K)
  \le 
  C D_k \Lambda^{-k}.
\end{equation}
Taking the $k$-th root here and letting $k\to \infty$, we
conclude that $\Lambda \le \Lambda_0(f)$ as desired.  
\end{proof}

It remains to prove part \ref{item:visexpfactors2}. For a given
expansion factor $\Lambda\in (1, \Lambda_0(f))$ we have to
construct a visual metric that satisfies \eqref{simmetric}.  We
have already encountered visual metrics with this local expansion
property; indeed, one can show that for the map $h$ in
Section~\ref{sec:int-frac-sph} the metric given by
\eqref{eq:def_visual_1} has the property \eqref{simmetric} with
$\Lambda=\Lambda_0(h)=2$.
 
The construction in the general case is much more 
difficult and  involved than the general construction of visual metrics in
Section~\ref{sec:exist-visu-metr}. We will first do this under
additional assumptions and then for the general case.

\subsection*{Construction of the metric under additional
  assumptions}
\label{sec:constr-metr-under}
Let $\CC\subset S^2$ be the Jordan curve with
$\post(f) \subset \CC$ used to define our cell decompositions
$\DD^n(f,\CC)$ and the quantities $D_n =D_n(f,\CC)$. We now assume in addition that $\CC$ is
$f$-invariant and that $\Lambda\in (1,\Lambda_0(f)]$ satisfies 
\begin{equation}\label{extraonL}
  \Lambda\le D_1=D_1(f,\CC).
\end{equation}

In this case, we will now construct a  visual metric $\varrho$ with expansion factor 
$\Lambda$ that 
satisfies \eqref{simmetric}. Note that $D_1 \le \Lambda_0(f)$ by Lemma~\ref{lem:submultexp} and that we do allow 
 $\Lambda=\Lambda_0(f)$ here if $D_1 = \Lambda_0(f)$. 
We first introduce some terminology.   

Recall from Definition~\ref{def:chains} that
a 
\defn{tile chain}\index{tile!chain}
$P$ is a finite sequence of tiles $X_1,\dots,
X_N$, where $X_j\cap X_{j+1}\neq \emptyset$ for $j=1, \dots, N-1$.
Here we do 
not 
require  the tiles to be  of the same levels. 

We define the {\em weight}  of a $k$-tile $X^k$ to be 
\begin{equation}
  w(X^k)\coloneqq \Lambda^{-k}, 
\end{equation}
 and the  \defn{$w$-length} of a tile  
 chain $P$ consisting of the tiles $X_1,\dots, X_N$ as  
\begin{equation*}
  \length_w(P)\coloneqq \sum_{j=1}^N w(X_j).
\end{equation*}

Now for  $x,y\in S^2$ we define
\begin{equation}
  \label{eq:defdF}
  \varrho(x,y)\coloneqq \inf_P\length_w(P),
\end{equation}
where the infimum is taken over all tile chains $P$ joining $x$ and
$y$.   Obviously,  such tile chains exist and the infimum can be taken over  simple tile
chains $P$.

\begin{lemma} \label{lem:C1}
The distance function $\varrho$ defined in \eqref{eq:defdF} is a visual metric for $f$  with expansion factor $\Lambda$. 
\end{lemma}

\begin{proof}
  Symmetry and the triangle inequality immediately follow  from the
  definition of $\varrho$. Obviously, we also  have  $\varrho(x,x)=0$ for $x\in S^2$.

  Let $x,y\in S^2$ with $x\ne y$ be arbitrary, and define 
  $m=m(x,y)=m_{f,\CC}(x,y)$ (see Definition~\ref{def:mxy}). Then there exist $m$-tiles $X$ and $Y$ with $x\in X$, $y\in Y$,  and $X\cap Y\ne \emptyset$. 
   So  $X,Y$ is a tile chain joining $x$
  and $y$, and thus 
  \begin{equation*}
    \varrho(x,y)\leq w(X) + w(Y)=2\Lambda^{-m}.
  \end{equation*} 
In order  to prove that $\varrho$ is a metric and is visual for $f$,    it remains to establish a lower bound $\varrho(x,y)\ge (1/C)\Lambda^{-m}$ 
 for a suitable constant $C$ independent of $x$ and $y$.  
  
  Pick $(m+1)$-tiles $X'$ and $Y'$ with $x\in X'$ and $y\in Y'$. Then 
  $X'\cap Y'= \emptyset$ by definition of $m$. Every tile chain joining 
  $x$ and $y$   
  contains a  simple tile chain  $P$ joining  $X'$ and $Y'$.  

Suppose  $P$  consists of the tiles $ X_1,\dots, X_N$. Let 
$k\in \N_0 $ be the largest level of any tile in $P$.    If $k\leq m+1$, then we get the favorable estimate
   \begin{equation}\label {favorable}
   \length_w(P)\ge \Lambda^{-k}\ge \Lambda^{-m-1}.
   \end{equation}

  Otherwise, $k>m+1$. We want to show that then  we can  replace the $k$-tiles 
  in $P$ with 
  $(k-1)$-tiles without increasing  the $w$-length of the tile
  chain (the construction is illustrated in Figure \ref{fig:pf_1_6}). 
  
  To see this, set $X_0=X'$, $X_{N+1}=Y'$, and  
  let  $X_i$,  where $1\le i\le N$,  be the first $k$-tile in
  $P$. Since $P$ is a simple tile
  chain joining $X'$ and $Y'$, the tile  $X_i$ is not contained 
  in $X_{i-1}$ and so it has to meet  $\partial X_{i-1}$. Since the level of $X_{i-1}$ is $<k$, we can find a $(k-1)$-edge $e\sub \partial X_{i-1}$ 
  with $e\cap X_i\ne \emptyset$. Here and below we use the fact that $\CC$ is $f$-invariant, and so cells of any level are subdivided by cells of higher  levels.  Every
  $(k-1)$-tile meets $e$ or is contained in the complement of the edge flower 
  $W^{k-1}(e)$ (see Lemma~\ref{lem:edgeflower}~\ref{item:flower_prop3}). Since  tiles of
  levels $\le k-1$ are subdivided into tiles of level  $k-1$, this
  implies also that every tile of level  $\le k-1$ meets $e$ or is contained in the complement of $W^{k-1}(e)$. 

  \ifthenelse{\boolean{nofigures}}{}{
    \begin{figure}
      \centering
      \begin{overpic}
        [width=12cm, 
        tics=20]{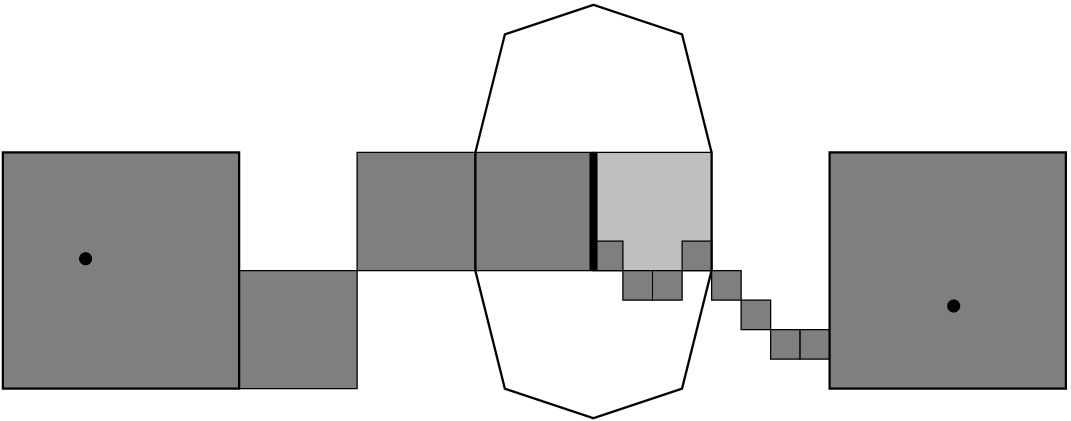}
        \put(9.4,14.7){$\scriptstyle{x}$}
        \put(90.4,10.5){$\scriptstyle{y}$}       
        \put(8,0){$\scriptstyle{X'=X_0}$}
        \put(27,0){$\scriptstyle{X_1}$}
        \put(47,12){$\scriptstyle{X_{i-1}}$}
        \put(55,12){$\scriptstyle{X_i}$}
        \put(55,26){$\scriptstyle{e}$}
        \put(60,26){$\scriptstyle{Z}$}
        \put(65,33){$\scriptstyle{W^{k-1}(e)}$}
        \put(67,15.5){$\scriptstyle{X_j}$}
        \put(73.5,3.5){$\scriptstyle{X_N}$}
        \put(85,0){$\scriptstyle{Y'=X_{N+1}}$}
      \end{overpic}
      \caption{Replacing $k$-tiles with  $(k-1)$-tiles.}
      \label{fig:pf_1_6}
    \end{figure}
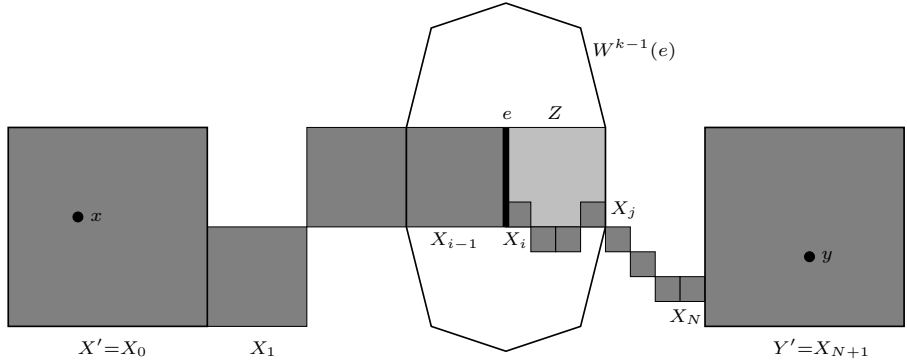
  }

  Now  $P$ is simple and so no
  tile in the ``tail'' $X_{i+1},\dots, X_N, X_{N+1}$  meets 
  $e$.  Let $j'\in \N$ be the largest number such that $i\le j'\le N$ and all 
  tiles $X_i, \dots, X_{j'}$ are $k$-tiles. Then $X_{j'+1}$ has level 
  $\le k-1$.  Since this tile does not meet $e$, it is contained in $S^2\setminus W^{k-1}(e)$, and so the tiles $X_i, \dots, X_{j'}$ form a chain of $k$-tiles joining $e$ and $S^2\setminus W^{k-1}(e)$. Let $j\in \N$ be the smallest number with  $i\le j\le j'$ such that $X_j$ meets the complement of $W^{k-1}(e)$. Then $X_i, \dots,  X_j$ 
  is a chain $P^k$ of $k$-tiles joining $e$ and $S^2\setminus W^{k-1}(e)$. 
 In particular, $P^k$ joins two disjoint $(k-1)$-cells as follows from the definition of an edge flower (see Definition~\ref{def:edgeflower}). 
  Moreover,  $X_j$ is the only tile in the chain $P^k$  that meets the complement 
  of $W^{k-1}(e)$.

Since $P^k$ joins disjoint $(k-1)$-cells, it follows from  Lemma~\ref{lem:flowerbds} that $P^k$   has at least   $D_1$ elements, and so  by \eqref{extraonL},  
   \begin{equation*}
    \length_w(P^k)\geq D_1 \Lambda^{-k}\geq \Lambda^{-k+1}.     
  \end{equation*}
  Let $Z$ be  the unique $(k-1)$-tile with $Z\supset X_j$. Then 
  $Z\cap X_{j+1}\ne \emptyset$. We also have $Z\cap e \ne \emptyset$.  For otherwise, $X_j\sub Z\sub S^2\setminus W^{k-1}(e)$. Then $j>i$ and  $X_{j-1}$ meets $X_j$ and so the complement of  $W^{k-1}(e)$ contradicting the definition of $j$. 
 So $Z\cap X_{i-1}\supset Z\cap e\ne \emptyset$. 
  Thus we can replace the subchain  $P^k$ of  $P$ with  the single $(k-1)$-tile $Z$ to obtain  a chain
  $P'$ joining  $X'$ and $Y'$. It satisfies 
  \begin{equation*}
    \length_w(P')=\length_w(P)-\length_w(P^k) + w(Z) \leq \length_w(P).  
  \end{equation*}
 By passing to a subchain of $P'$ we can find a simple tile chain
 $P''$  
 joining $X'$ and $Y'$ that contains fewer $k$-tiles than $P$ and satisfies $\length_w(P'')\le \length_w(P)$. 
 
%


 Continuing this process, we can remove all $k$-tiles from the
 tile chain joining $X'$ and $Y'$ without increasing its
 $w$-length. If $k-1>m+1$, we can repeat the process and remove
 the $(k-1)$-tiles without increasing the $w$-length, etc.  In
 the end, we obtain a tile chain $\widetilde P$ joining $X'$ and
 $Y'$ that contains no tiles of levels $>m+1$ and satisfies
 $\length_w(\widetilde P)\le \length_w(P)$.  Thus
  \begin{equation*}
    \length_w(P)\geq \length_w (\widetilde P)\ge \Lambda^{-m-1}.
  \end{equation*} 
  
  This together with the previous estimate \eqref{favorable} implies 
  $$\varrho(x,y)\ge \Lambda^{-m-1}. $$
  This is an inequality as desired, and so $\varrho$ is indeed a visual
  metric with expansion factor $\Lambda$.  \end{proof}

 \begin{lemma}\label{lem:C2}  The visual metric $\varrho$ as defined in \eqref{eq:defdF}
  has the expansion property \eqref{simmetric}.  
  \end{lemma}
  
  \begin{proof} 
  Since $\varrho$ is a visual metric for $f$, 
  it induces the given topology on $S^2$ (see
  Proposition~\ref{prop:visualsummary}~\ref{item:vistop}). 
 So in  the ensuing proof, we can  rely on the usual 
 characterization of open subsets of $S^2$ in terms of 
  metric balls for $\varrho$. 
  
  We first show that 
\begin{equation}\label{eq:uppexp}
    \varrho(f(x),f(y)) \leq\Lambda \varrho(x,y),
  \end{equation}
  for all $x,y\in S^2$ with $\varrho(x,y)<1$. 
    
  Indeed, suppose $x,y\in S^2$ are arbitrary points with $\varrho(x,y)<1$.  
  Let $P$ be an arbitrary tile chain that joins $x$ and $y$ and suppose that it consists  of the tiles $X_1, \dots, X_N$.
 We may assume in addition that  $P$ satisfies $\length_w(P)<1$. Then $P$ does not contain $0$-tiles and hence $f(X_1), \dots, f(X_N)$ is a tile chain  joining $f(x)$ and $f(y)$. Denoting  the latter chain by $f(P)$, we have $$\length_w (f(P))=\Lambda \length_w(P). $$
 Taking the infimum over all such tile chains $P$, we obtain   the desired inequality \eqref{eq:uppexp}. 
  
For an inequality in the other direction
we now consider two cases for $x\in S^2$. 

\smallskip
{\em Case 1:} $x\notin \crit(f).$  Then we can find an open neighborhood $U$ of $x$ such that $f|U$ is a homeomorphism of $U$ onto $U'\coloneqq f(U)$.  Then $U'$ is an open set containing
$f(x)$.  We can choose $\eps>0$ and $\delta\in (0,1)$ such that
$B_\varrho(x,\delta)\sub U$, $B_\varrho(f(x),\eps)\sub U'$, and
$f(B_\varrho(x,\delta))\sub B_\varrho(f(x),\eps)$.

Define $U_x=B_\varrho(x,\delta)$, and let $y\in U_x$ be
arbitrary. 
Then
$\varrho(f(x), f(y))<\eps$. Consider a
tile chain $P'$ joining $f(x)$ and $f(y)$ whose $w$-length is
close enough to $\varrho(f(x), f(y))$ so that
$\length_w(P')<\eps$. By definition of the metric $\varrho$, for
every point $z$ that belongs to a tile in $P'$, we have
$\varrho(f(x),z)\le \length_w(P')<\eps$. Hence $P'$ lies in
$B_\varrho(f(x), \eps)\sub U'$.
    
It follows that $(f|U)^{-1}$ is defined on every tile $X'$ in
$P'$; so
by Lemma~\ref{adhoc}~\ref{item:adhoc1} the Jordan
region $X=(f|U)^{-1}(X')$ is a tile contained in $U$. If $k$ is
the level of $X'$, then $k+1$ is the level of $X$.  By
considering these images of tiles in $P'$ under $(f|U)^{-1}$, we
get a tile chain $P$ joining $x$ and $y$ with
$\length_w (P) = (1/\Lambda) \length_w(P')$. Taking the infimum
over such $P'$, we obtain
\begin{equation}\label{uppdbd}
  \varrho(x,y)\leq (1/\Lambda) \varrho(f(x),f(y)).
\end{equation}
   
  \smallskip
{\em Case 2:}    
$x\in\crit(f)$.
Then $x\in f^{-1}(\post(f))$, and so $x$  is a $1$-vertex. Consider the flower $U=W^1(x)$, and its image $U'=f(W^1(x))=W^0(f(x))$. These are open neighborhoods of $x$ and $f(x)$, respectively, and the map 
$f|U\setminus\{x\}$  is an (unbranched) covering map of $U\setminus\{x\}$ onto $U'\setminus \{f(x)\}$ (this follows from the last part of  Lemma~\ref{lem:mapflowers}~\ref{item:mapflowers1}).  
Again we can choose $\eps>0$ and $\delta\in(0,1)$ such that $B_\varrho(x,\delta)\sub U$, 
    $B_\varrho(f(x),\eps)\sub U'$,  and $f(B_\varrho(x,\delta))\sub B_\varrho(f(x),\eps)$.

   Define $U_x=B_\varrho(x,\delta)$, and  let $y\in U_x$ be arbitrary. In order to show \eqref{uppdbd}, we may assume $x\ne y$. Then 
   $\varrho(f(x), f(y))<\eps$ and $f(x)\ne f(y)$.  Consider a tile chain $P'$
    joining  $f(x)$ and $f(y)$ consisting of tiles $X'_1,\dots, X'_N$.  We can make the further assumptions that $X'_1$ is the only tile in this chain that contains $f(x)$ and that $\length_w(P')$ is close enough to  $\varrho(f(x), f(y))$ such that  $\length_w (P')<\eps$. As before, this implies that $P'$ lies in $U'$.  
    We now   choose a path $\ga\:[0,N]\ra U'$ with the following properties:

    \begin{enumerate} 
    
    \item
      $\ga(0)=f(x)$, $\ga(N)=f(y)$, and $\ga(t)\ne f(x)$ for $t\ne 0$.   
    
    \item
      $\ga([i-1,i])\sub X'_i$ for $i=1, \dots, N$.  
    
    \item
      $\ga(i-1/2)\in \inte(X'_i)$ for $i=1, \dots, N$.
    
    \end{enumerate}
 Since the tiles $X'_i$ are Jordan regions,  such a path $\ga$ can easily be obtained by first running in $X'_1$ from $f(x)$  to an interior point of $X'_1$, then in $X_1'$ to a point in $X'_1\cap X'_2$, then in $X_2'$ to an interior point  of $X'_2$, etc., and finally in $X'_N$  to $f(y)\ne f(x)$. Since $X'_1$ is the only tile in $P'$ containing $f(x)$, this can be done so that the path never meets $f(x)$ except in its initial point. 

There exists a lift $\alpha$ of this path by $f$  with endpoints $x$ and  $y$, i.e., a path $\alpha\:[0,N]\ra U$ with $\alpha(0)=x$, 
$\alpha(N)=y$, and $f\circ \alpha=\ga$. To obtain $\alpha$, lift
$\ga|(0,N]$ by  the covering map $f|U\setminus\{x\}$  such that the lift ends  at $y$ (see Lemma~\ref{lem:liftsofcov} and the subsequent discussion), and note that the lift  has a unique continuous extension to $[0,N]$ by choosing $x$ to be its initial point. 

Using this lift $\alpha$, we can construct a lift of our tile chain $P'$ as follows. Consider a tile $X'_i$ in $P'$ and 
let $k_i$ be its level. Set $p_i\coloneqq \alpha(i-1/2) $ and $p_i'\coloneqq \gamma(i-1/2)$.
Then $f(p_i)=p'_i\in \inte(X_i')$. By Lemma~\ref{adhoc}~\ref{item:adhoc2} there exists a unique $(k_i+1)$-tile $X_i$ with $p_i\in X_i$ and $f(X_i)=X_i'$. 

Note that 
$$\gamma((0,N])\sub U'\setminus\{f(x)\}=W^0(f(x))\setminus \{f(x)\}\sub S^2\setminus\post(f)$$ and that the map 
$$f\: S^2\setminus f^{-1}(\post(f)) \ra S^2\setminus \post(f)$$ is a
covering map (see Lemma~\ref{lem:brcovcov}). This implies that $\alpha|[i-1,i]$ is the unique lift of $\gamma|[i-1,i]$ with $\alpha(i-1/2)=p_i$ (see 
Lemma~\ref{lem:liftsofcov}~\ref{item:liuniq}). On the other hand, the path  $\beta_i=(f|X_i)^{-1}\circ (\ga|[i-1,i])$ is also a lift of $\ga|[i-1,i] $ by  $f$ with $\beta_i(i-1/2)=p_i$ by definition of $X_i$. 
 Hence  $\beta_i=\alpha|[i-1,i]$ and so 
$\alpha([i-1,i])\sub X_i$. It follows that $x=\alpha(0)\in X_1$, $y=\alpha(N)\in X_N$, and $X_i\cap X_{i+1}\supset \{\alpha(i)\}\ne \emptyset$
for $i=1, \dots, N-1$.

Therefore, the tiles $X_1, \dots, X_N$ form a tile chain $P$
joining $x$ and $y$. The level of each tile in $P$ exceeds the
level of the corresponding tile in $P'$ by exactly $1$. Hence
$\length_w (P) = (1/\Lambda) \length_w (P')$. Taking the infimum
over such $P'$, we again obtain inequality \eqref{uppdbd}

Combining \eqref{eq:uppexp} and  \eqref{uppdbd}, we see that every point $x\in S^2$ has a neighborhood $U_x$ such that \eqref{simmetric} holds. 
\end{proof}

This concludes the proof for the existence of   the  visual metric $\varrho$ with the desired properties as in Theorem~\ref{thm:visexpfactors1}~\ref{item:visexpfactors2} 
under the additional assumptions   that $\CC$ is $f$-invariant and that \eqref{extraonL} holds.
 We now consider the general case.
 
 \begin{proof}
   [Proof of
   Theorem~\ref{thm:visexpfactors1}~\ref{item:visexpfactors2}]
 Suppose that $1 < \Lambda < \Lambda_0(f)$. We can choose an
   iterate $F=f^n$ of $f$ such that $F$ has an $F$-invariant
   Jordan curve $\CC\sub S^2$ with $\post(f)=\post(F)\sub \CC$
   (Theorem~\ref{thm:main}). Note that
   $D_k(f,\CC)^{1/k}\to \Lambda_0(f)>\Lambda$ as $k\to \infty$  by
   Proposition~\ref{prop:exp}. Hence if $n$ is sufficiently
   large, which we may assume by passing to an iterate of $F$, we
   also have
$$ D_1(F,\CC) = D_n(f,\CC)\ge \Lambda^n. $$

This means that $F$ is an expanding Thurston map  that
satisfies  condition \eqref{extraonL} for the $F$-invariant
Jordan curve $\CC$. This  allows us to construct a metric for $F$ as discussed above. We call this metric $d$ in order to
distinguish it from the metric $\varrho$ that we are trying to
find for $f$. Then $d$ is a visual metric for $F$ with expansion
factor $\Lambda^n$, and for each $x\in S^2$ there exists an open
neighborhood $U_x$ of $x$ such that
\begin{equation}\label{rhoexp}
d(F(x), F(y))=d(f^n(x), f^n(y))=\Lambda^n d(x,y)
\end{equation}
for all $y\in U_x$. 

We now define $\varrho$ as 
\begin{equation}
  \label{eq:defdf}
 \varrho(x,y)=\frac{1}{n}\sum_{i=0}^{n-1}\Lambda^{-i} d(f^{i}(x), f^{i}(y))
\end{equation}
for $x,y\in S^2$. It is clear that $\varrho$ is a metric on $S^2$. 

Property  \eqref{simmetric} for the metric  $\varrho$ follows from the corresponding property 
\eqref{rhoexp} for $d$  with the same sets $U_x$, $x\in S^2$; indeed,  
if $x\in S^2$ and $y\in U_x$ then by \eqref{rhoexp}  we have
\begin{align*}
  \varrho(f(x),f(y)) &= \frac{1}{n} \sum_{i=0}^{n-1}\Lambda^{-i} d(f^{i+1}(x), f^{i+1}(y))
  \\
  & = \frac{1}{n}\biggl(\Lambda \sum_{i=0}^{n-2}\Lambda^{-(i+1)} d(f^{i+1}(x), f^{i+1}(y))+ \Lambda d(x,y)\biggr)
  \\
  & = \Lambda \frac{1}{n}  \sum_{i=0}^{n-1}\Lambda^{-i} d(f^{i}(x), f^{i}(y)) 
  \,  =\,  \Lambda \varrho(x,y).
\end{align*}

It remains to show that $\varrho$ is a visual metric for $f$ with expansion factor $\Lambda$. 
Let $m=m_{f,\CC}$ and $m_F=m_{F,\CC}$ be as in Definition~\ref{def:mxy}. Since 
$d$ is a visual metric for $F$ with expansion factor $\Lambda^n$, we have 
$$ d(x,y)\asymp \Lambda^{-nm_F(x,y)}\asymp \Lambda^{-m(x,y)}$$ for
all $x,y\in S^2$ by Lemma~\ref{lem:mprops}~\ref{item:mprops4}.
Hence 
$$\varrho(x,y)\ge 
     \frac{1}{n} d(x,y)\gtrsim \Lambda^{-m(x,y)} . $$
Moreover, by  Lemma~\ref{lem:mprops}~\ref{item:mprops2} we have 
$$ m(f^i(x), f^i(y))\ge m(x,y)-i$$
and so 
$$ d(f^i(x), f^i(y))\asymp \Lambda^{-m(f^i(x),f^i(y))}\le  
\Lambda^i\Lambda^{-m(x,y)}$$ 
for all $i\in \N_0$. Hence 
$$\varrho(x,y)\lesssim \frac 1n \sum_{i=0}^{n-1}\Lambda^{-m(x,y)}
= \Lambda^{-m(x,y)}. $$

It follows that $\varrho(x,y)\asymp \Lambda^{-m(x,y)}$ for all $x,y\in S^2$,
where $C(\asymp)$ is  independent of $x$ and $y$. This  shows that $\varrho$ is a visual metric for $f$ with expansion factor
$\Lambda$. 
\end{proof}

We conclude the chapter with an example showing that for an expanding Thurston map $f$ one can in general not expect 
 the existence of a visual metric with expansion factor 
  $\Lambda=\Lambda_0(f)$.  

\begin{ex}\label{ex:notattained} The example is a Latt\`es-type map as in  Example~\ref{ex:lattes_type}. 
 We consider the crystallographic group  
$G$ consisting of all isometries $g$ on $\R^2$ of the form $u\mapsto g(u)=\pm u+\gamma$, where $\gamma\in\Z^2$. 
Then the quotient space $S^2\coloneqq \R^2/G$ is a $2$-sphere.
Let $\Theta\: \R^2\ra S^2=\R^2/G$ be the quotient map. We know that $\Theta$ is induced by $G$ and so $\Theta(u_1)=\Theta(u_2)$ for $u_1,u_2\in \R^2$ if and only if there exists $g\in G$ such that $u_2=g(u_1)$.

 As in Example~\ref{ex:lattes_type}, 
one may view
 $S^2=\R^2/G$ as a pillow 
 obtained   
by folding the rectangle $R=[0,1]\times [0,1/2]$ along the line 
$\ell=\{(x,y)\in \R^2: x=1/2\}$ and identifying the boundaries of the squares 
$S=[0,1/2]\times [0,1/2]$   and $S' = [1/2,1]\times [0,1/2]$
under this operation. In particular, $\CC\coloneqq \Theta(\partial S)=\Theta(\partial S')$ is  a Jordan curve  containing the four critical
values of $\Theta$, which are  
the four vertices of the pillow.    

Let 
$$ A= \left(\begin{array}{cc} 2 & 2\\
0 & 2\end{array}\right).  $$ 
Then  
for $n\in \N_0$ we have 
\begin{equation} \label{eq:Apower}
A^n= \left(\begin{array}{cc} 2^n & n2^{n}\\
0 & 2^n\end{array}\right)  \text{ and }
A^{-n}= \left(\begin{array}{cc} 2^{-n} & -n2^{-n}\\
0 & 2^{-n}\end{array}\right).\end{equation}

We consider the map $A\: \R^2\ra \R^2$, $u\in \R^2\mapsto Au$,  
given by left-multiplication  of $u\in \R^2$ (considered as a column vector)
by the matrix $A$.  For simplicity we use the same notation for the matrix $A$ and this  linear map on $\R^2$. 

Then there exists a  unique Latt\`es-type map $f\: S^2 \ra S^2$  such that the diagram
\begin{equation*}
  \xymatrix{
    \R^2 \ar[r]^A \ar[d]_{\Theta} &
    \R^2 \ar[d]^{\Theta}
    \\
    S^2 \ar[r]^f & S^2
  }
\end{equation*}
commutes (see Proposition~\ref{prop:2222}).  The map $f$ has signature $(2,2,2,2)$ and  its four postcritical points are the critical values of $\Theta$. In particular,
 $\post(f)\sub \CC$. Proposition~\ref{prop:expLattType} implies that the Thurston map 
 $f$ is expanding.


Since $A$ induces a map on the quotient $\R^2/G$, it is $G$-equivariant and so we have $A\circ g\circ A^{-1}\in G$ whenever 
$g\in G$ (Lemma~\ref{lem:f_groupdescend}). This implies that if $n\in \N_0$ and  $\alpha\in \frac12 \Z^2$, then $\Theta$ is injective on 
the parallelogram  $A^{-n}(\alpha +S)$. Indeed, if $u_1,u_2\in  \alpha +S$ and 
$ \Theta(A^{-n}(u_1))=\Theta(A^{-n}(u_2))$, then there exists $g\in G$ such that 
$A^{-n}(u_2) =g(A^{-n}(u_1))$. Then  $u_2=h(u_1)$ with $h\coloneqq A^n\circ g\circ A^{-n}\in G$, and so $\Theta(u_2)=\Theta(u_1)$; since $\Theta$ is injective on the square $\alpha +S$, we conclude  that $u_2=u_1$. The injectivity of $\Theta$ on  $A^{-n}(\alpha +S)$ follows. 

This implies that  the set $X^n=\Theta(A^{-n}(\alpha +S))$ is a Jordan region whenever   $n\in \N_0$ and $\alpha\in \frac12 \Z^2$. The $n$-tiles for $(f,\CC)$ are precisely the sets $X^n$ of this form. This is clear for $n=0$. If $n\in \N_0$ is arbitrary, then $f^n\circ \Theta =\Theta\circ A^n$. 
 Since $\Theta|(\alpha+S)$ is a homeomorphism of $\alpha+S$ onto the $0$-tile $\Theta(\alpha+S)$, it follows that $f^n|X^n$ is a homeomorphism of the Jordan region $X^n$  onto a $0$-tile. 
So by Lemma~\ref{adhoc}~\ref{item:adhoc1}, the set $X^n$ is indeed an $n$-tile. Since 
these sets  $X^n$ cover $\Theta(\R^2)=S^2$, there are no other $n$-tiles.


It follows from \eqref{eq:Apower} that  each set  $A^{-n}(\alpha+S)\subset \R^2$ is a parallelogram congruent to  the parallelogram $P_n =\{ xu_n+y v_n: 0\le x,y\le 1\}$ 
spanned by the vectors  $u_n=\frac 1{2^{n+1}}(1,0)$ and 
$v_n=\tfrac 1{2^{n+1}} (-n, 1)$.   Thus $\diam(A^{-n}(\alpha+S))\asymp n 2^{-n}$, where $C(\asymp)$ is
independent of $n$. This implies that if we equip the pillow
$S^2$ with the locally Euclidean metric (obtained by pushing the
Euclidean metric forward by $\Theta$), then for each $n$-tile
$X^n$ we have $\diam(X^n) \asymp n2^{-n}$ (so 
by Proposition~\ref{lem:expoexp}~\ref{item:expoex2} 
this
metric on $S^2$ is {\em not} a visual metric for $f$).  We conclude
that $D_n\coloneqq D_n(f,\CC)\gtrsim 2^n/n$ for $n\in \N$.

\pagebreak
If $X_i\coloneqq \Theta(A^{-n}(\alpha_i+S))$ with  $\alpha_i=(0,-i/2)$  for $i=1, \dots , N\coloneqq 
\lceil 2^n/n\rceil$, then $X_1, \dots, X_N$ is a chain of $n$-tiles joining opposite sides of $\CC$. Hence $D_n\le N\lesssim 2^{-n}/n$ for $n\in \N$.  
It  follows that 
$D_n\asymp 2^n/n$,  and so 
$$\Lambda_0(f)=\lim_{n\to \infty} D_n^{1/n}=2. $$ 

If $\Lambda$ is the expansion factor of a visual metric, then for all $n\in \N$ we
must have $1\lesssim D_n\Lambda^{-n}\asymp 2^n\Lambda^{-n}/n
$ (see \eqref{eq:DkLged} in the proof
Theorem~\ref{thm:visexpfactors1}~\ref{item:visexpfactors1}).  
It follows that a visual metric for $f$ with expansion factor
$\Lambda =\Lambda_0(f)=2$ does not exist.
\end{ex}

\ifthenelse{\boolean{singlechapter}}{

%
%


\chapter{The measure of maximal entropy}
\label{cha:measure}

In this chapter  we investigate  the measure of maximal 
entropy  of an expanding  Thurston map. We will first review some definitions and the necessary background from measure-theoretic dynamics in Section~\ref{sec:reviewmdyn}. Our goal in Section~\ref{sec:meamaxemt} is then to prove  the following statement.

\begin{theorem} 
  \label{thm:maxentr0}
  \index{measure!of maximal entropy}
  \index{n@$\nu_f$|textbf} 
Let $f\: S^2 \ra S^2$ be an expanding Thurston map. Then there exists a unique measure $\nu_f$ 
of maximal entropy for $f$. The map $f$ is mixing for $\nu_f$. 
\end{theorem}
Since mixing implies ergodicity, it follows  that $f$ is ergodic with respect to $\nu_f$.

On our way to prove the previous theorem, we will be able to 
compute the topological entropy $h_{top}(f)$ of $f$. 

\begin{cor}\label{cor:topent} Let $f\: S^2 \ra S^2$ be an expanding Thurston map. Then $h_{top}(f)=\log(\deg(f))$. 
\end{cor}

A consequence of the uniqueness part in Theorem~\ref{thm:maxentr0} is that the measures of maximal entropy of an expanding Thurston map and any of its iterates agree.

\begin{cor}\label{cor:momeit}  
Let $f\: S^2 \ra S^2$ be an expanding Thurston map.  Then for each $n\in \N$ we have $\nu_f=\nu_{f^n}$ for the unique measures of maximal entropy of $f$ and $f^n$. 
\end{cor} 

The corresponding statements of Theorem~\ref{thm:maxentr0} and
Corollary~\ref{cor:topent} for rational maps (not necessarily
postcritically-finite) 
where proved by Lyubich in \cite{Ly83}. 
For expanding Thurston maps without periodic critical points,
Theorem~\ref{thm:maxentr0} can be derived from general results
due to Ha\"\i ssinsky and Pilgrim \cite{HP}. 

We will present a
different approach.
We consider an iterate  $F=f^n$  of our given expanding Thurston map 
$f$ 
that has an $F$-invariant Jordan curve $\CC$ 
with $\post(F)\sub \CC$. Then in the  cell decomposition $\DD^k(F,\CC)$, $k\in \N_0$,  generated by $F$ and $\CC$ each cell is subdivided by cells of higher levels. 
This will allow us to construct a specific $F$-invariant probability measure $\nu_F$ that assigns to each $k$-tile  mass proportional to $\deg(F)^{-k}$, where the proportionality factor only depends on the color of the tile
(see Proposition~\ref{prop:exmeasure}).  

By using coverings by tiles, it is also easy to obtain the estimate 
$h_{top}(F)\le \log(\deg(F))$ (see the proof of Lemma~\ref{lem:topent}). On the other hand, the $1$-tiles in $\DD^1(F,\CC)$ form a measurable partition of $S^2$ that generates (in a suitable sense) the Borel $\sigma$-algebra on $S^2$ (see Lem\-ma~\ref{lem:generator}). This  allows us to   explicitly compute  the  measure-theoretic entropy 
$h_{\nu_F}(F)$ of $F$ with respect to $\nu_F$  as $h_{\nu_F}(F)=\log(\deg(F))$ (see the proof of Proposition~\ref{prop:exmeasure}). It follows that $\nu_F$ is a measure of maximal entropy for $F$. 
One then shows that this measure $\nu_F$ is actually $f$-invariant and the unique measure of maximal entropy for $f$; this is formulated in  Theorem~\ref{thm:nuF} which immediately implies Theorem~\ref{thm:maxentr0}.

Our  main point here  is to give a rather concrete and elementary description of the measure of maximal entropy of an expanding Thurston map and to establish some of its basic properties.   We will not touch upon many interesting related questions such as equidistribution 
of preimages and periodic points or more general measures such as equilibrium measures for suitable potentials. These subjects are thoroughly investigated in \cite{Li1, Li2}.

\section{Review of measure-theoretic dynamics} \label{sec:reviewmdyn}

In this section we briefly discuss the necessary concepts related to topological and measure-theoretic entropy.  For more background  on these topics see \cite{KH,Wa}. 

In the following,  $(X,d)$ is  a compact metric space, and
$g\:X\ra X$ is a continuous map. 
For $n\in \N$  and $x,y\in X$ we define 
\begin{equation}\label{eq:defdng}
d_g^n(x,y)=\max\{d(g^k(x), g^k(y)): k=0, \dots, n-1\}.\end{equation} 
Then $d_g^n$ is a metric on $X$. Let $D(g,\eps,n)$ be the minimal cardinality  of a family of  
subsets of $X$  whose $d_g^n$-diameter is at most $\eps>0$ and whose union is equal to  $X$.
One can show that the limit
$$h(g,\eps)\coloneqq \lim_{n\to\infty}\frac 1n \log(D(g,\eps,n))$$ 
exists \cite[Lemma~3.1.5]{KH}. 
Obviously, the quantity
$h(g,\eps)$ is non-increasing in $\eps$. One defines the 
{\em topological 
entropy}\index{topological entropy}\index{entropy!topological}\index{h top@$h_{top}$}
of $g$ (see \cite[Section~3.1.b]{KH}) as  
$$h_{top}(g)\coloneqq \lim_{\eps\to 0} h(g,\eps)\in[0,\infty]. $$
If one uses another metric $d'$ on $X$, then one obtains the
same quantity for $h_{top}(g)$ if $d'$ induces the same topology
on $X$ as $d$ \cite[Proposition~3.1.2]{KH}. 
The topological entropy is also  well-behaved under iteration. Indeed, 
if  $n\in \N$, then $h_{top}(g^n)=n h_{top}(g) $ 
\cite[Proposition~3.1.7~(3)]{KH}.

We denote by ${\mathcal B}$  the $\sigma$-algebra of all Borel
sets on $X$.  A measure  on $X$ is understood to be a Borel
measure, i.e., one defined on  ${\mathcal B}$. We call a
measure 
$\mu$ on $X$ $g$-{\em invariant} if 
\index{invariant!measure}
\index{measure!invariant}    
\index{f-invariant@$f$-invariant!measure}
\begin{equation}
  \label{invariance}
  \mu(g^{-1}(A))=\mu(A)
\end{equation} for all $A\in {\mathcal B}$. Note that by continuity of
$g$, we have $g^{-1}(A)\in \mathcal{B}$ whenever $A\in
\mathcal{B}$. We denote by $\mathcal{M}(X,g)$\index{M@$\mathcal{M}(X,g)$} the set of all
$g$-invariant Borel probability measures on $X$.   
 
If $\mu$ is a probability measure on a compact metric space 
$X$, then it is 
{\em regular}.\index{measure!regular}\index{regular measure}
This means that for every $\eps>0$ and every  Borel set $A\sub X$ there exists a compact set $K\sub A$ with $\mu(A\setminus K)<\eps$  
({\em inner regularity})  and an open set $U\sub X$ with $A\sub U$ and $\mu(U\setminus A)<\eps$ ({\em outer regularity}). See \cite[Theorem~2.18]{Ru} for a more general result  that contains this statement as a special case. 

A {\em semi-algebra} ${\mathcal S}$  
is a family of subsets of $X$ 
satisfying the following conditions: (i) $\emptyset \in {\mathcal S}$, (ii) $A\cap B\in {\mathcal S}$, whenever $A,B\in {\mathcal S}$, 
and (iii)  $X\setminus A$ is a finite union of disjoint sets in ${\mathcal S}$, whenever $A\in  {\mathcal S}$. A semi-algebra ${\mathcal S}$ 
{\em generates} a $\sigma$-algebra ${\mathcal A}$ on $X$ if ${\mathcal A}$ is the smallest 
$\sigma$-algebra containing ${\mathcal S}$. 

Let $\mathcal{S}$ be a semi-algebra generating ${\mathcal
  B}$. If $\mu$ and $\nu$ are two probability measures on $X$ and
$\mu(A)=\nu(A)$ for all $A\in \mathcal{S}$, then
$\mu=\nu$. Similarly, in order to show that a probability measure $\mu$ is
$g$-invariant it is enough to   verify \eqref{invariance} for
all sets $A$ in ${\mathcal S}$ (see  \cite[proof of
Theorem~1.1, p.~20]{Wa}  for the simple argument on how to verify these statements).

Let  $\mu\in  \mathcal{M}(X,g)$. Then we say that $g$  is  {\em ergodic}\index{ergodic}\index{measure!ergodic}
for $\mu$ (or $\mu$ is {\em ergodic} for $g$) if  for each set
$A\in {\mathcal B}$ with $g^{-1}(A)=A$  we have  
$\mu(A)=0$ or $\mu(A)=1$. The map $g$ is called 
{\em mixing}\index{mixing}  
for $\mu$ if 
\begin{equation}\label{mixing}
\lim_{n\to \infty} \mu(g^{-n}(A)\cap B)=\mu(A)\mu(B)
\end{equation}
 for all $A,B\in {\mathcal B}$. It is easy to see that if  $g$ is mixing for $\mu$, then $g$ is also ergodic.  
 
 To establish mixing, one only has to verify \eqref{mixing} for
 sets $A$ and $B$ in a semi-algebra generating $\mathcal{B}$
 (\cite[Theorem~1.17~(iii)]{Wa}; 
note that the terminology in \cite{Wa} slightly differs from ours).  
 If $\mu,\nu\in  \mathcal{M}(X,g)$,  $g$ is ergodic for $\mu$,
  and $\nu$ is absolutely continuous with respect to $\mu$,
 then $\nu=\mu$ \cite[Remark~(1), p.~153]{Wa}. 

Our next goal is to define the measure-theoretic entropy of $g$ for a measure 
$\mu$.  We will follow \cite[Section~4.3]{KH} with slight differences 
in notation and terminology (see also \cite[Chapter~4]{Wa}).

Let $\mu\in  \mathcal{M}(X,g)$. A 
{\em measurable partition}\index{measurable partition}\index{partition!measurable}
$\xi$ for $(X,\mu)$ is a countable  collection $\xi=\{A_i:i\in
I\}$ of sets in $\mathcal{B}$ such that  
$\mu(A_i\cap A_j)=0$ for $i,j\in I$, $i\ne j$,  and 
$$ \mu\biggl(X\setminus \bigcup_{i\in I} A_i\biggr)=0. $$ Here $I$ is a countable (i.e., finite or countably infinite) index set. 
The {\em symmetric difference}\index{$\norm{D2}$@$\bigtriangleup$}\index{symmetric difference} of two sets $A,B\sub X$ is defined as 
$$A\bigtriangleup B=(A\setminus B)\cup (B\setminus A).$$ 
Two measurable partitions $\xi$ and $\eta$ for $(X,\mu)$ are called
{\em equivalent}\index{partition!equivalent}\index{equivalent partition}
if there exists a bijection between the sets of positive measure in $\xi$ and the sets of   positive measure in $\eta$ such that corresponding sets have 
a symmetric difference of  $\mu$-measure zero. Roughly speaking, this means that the partitions are the same up to sets of measure zero.

Let  $\xi=\{A_i:i\in I\}$ and $\eta=\{B_j:j\in J\}$ be  measurable partitions of $(X,\mu)$. 
Then
$$\xi \vee \eta\coloneqq \{ A_i\cap B_j: i\in I,\,  j\in J\} $$
is also a measurable partition,  called the {\em join} of $\xi$ and $\eta$. The join of finitely many measurable partitions is defined similarly. 

Let  
$$g^{-1}(\xi)\coloneqq \{g^{-1}(A_i): i\in I\} $$ 
and for $n\in \N$ define 
\begin{equation}
  \label{eq:def_xi_ng}
  \xi^{n}_g
  \coloneqq
  \xi\vee g^{-1}(\xi)\vee \dots \vee g^{-(n-1)}( \xi). 
\end{equation}

The {\em entropy}\index{entropy!of partition} 
of $\xi$ is
\begin{equation*}
  H_\mu(\xi)
  \coloneqq
  \sum_{i\in I} \mu(A_i)\log(1/\mu(A_i))\in [0,\infty].  
\end{equation*}
Here it is understood that the function  $\phi(x)=x\log(1/x)$ is
continuously extended to $0$ by setting  $\phi(0)=0$.   

One  can  show that if $H_\mu(\xi)<\infty$, then for a given map $g$ the quantities
$H_\mu(\xi^{n}_g)$, $n\in \N_0$, are 
{\em subadditive}\index{subadditive}
 in the sense  that  $$H_\mu(\xi^{n+k}_g)\le H_\mu(\xi^{n}_g)+H_\mu(\xi^{k}_g)$$ for all  $k,n\in \N_0$ 
\cite[Proposition~4.3.6]{KH}. 
This implies 
  that
$$h_\mu(g,\xi)\coloneqq \lim_{n\to\infty}\frac 1n
H_\mu(\xi^n_g)\in [0,\infty)$$ exists 
and
we have 
\begin{equation}\label{eq:linhginf}
h_\mu(g,\xi)=\inf_{n\in \N}\frac 1n H_\mu(\xi^n_g) 
\end{equation}
(\cite[Theorem~4.9]{Wa}; 
see also the proof of
Lemma~\ref{lem:submultexp}). 
The quantity $h_\mu(g, \xi)$ is called the {\em
  (measure-theoretic) entropy of $g$ relative to} $\xi$. The 
{\em (measure-theoretic) entropy}\index{measure-theoretic entropy|textbf}\index{entropy!measure-theoretic|textbf}\index{h mu@$h_\mu$|textbf}
 of $g$ for $\mu$ is defined as 
 \begin{align}
   h_\mu(g)=\sup\{ h_\mu(g,\xi): 
   {}&\xi \text{ is a measurable partition} 
   \\ \notag
   &\text{of } (X,\mu)
   \text{ with } H_\mu(\xi)<\infty\}.
 \end{align}
In this  definition it is actually enough to take the supremum over all finite measurable partitions $\xi$ (this easily follows from 
``Rokhlin's inequality'' 
\cite[Proposition~4.3.10 (4)] {KH}).

We call a finite  measurable partition $\xi$ a {\em generator}\index{generator}    
for $(g,\mu)$ if 
the following condition is true: Let $\mathcal{A}$ be the
smallest $\sigma$-algebra containing all sets in the partitions
$\xi_g^n$, $n\in \N$. Then we require that for each Borel set
$B\in \mathcal{B}$ there exists a set $A\in \mathcal{A}$ such
that 
$\mu(A\bigtriangleup B)=0$.

 If for every set $B\in \mathcal{B}$ and for  every $\eps>0$, there exists $n\in \N$
and a union $A$ of sets in  $\xi_g^n$ with $\mu(A\bigtriangleup B)<\eps$, then $\xi$ is a generator for $(g,\mu)$. 
If $\xi$ is a generator, then   $h_\mu(g)=h_\mu(g,\xi)$  by the
Kolmogorov-Sinai theorem 
\cite[Theorem.~4.17]{Wa}.   

If $\mu\in {\mathcal M}(g,X)$ and  $n\in \N$, then 
\cite[Proposition~4.3.16~(4)]{KH}
\begin{equation}
  \label{eq:hfhF}
  h_\mu(g^n)=n h_\mu(g).
\end{equation}

If $\alpha\in [0,1]$ and   $\nu\in {\mathcal M}(g,X)$ is another
measure,  then 
\cite[Theorem~8.1]{Wa}
$$ h_{\alpha \mu+(1-\alpha) \nu}(g)=\alpha h_{\mu}(g)+(1-\alpha)h_{\nu}(g). $$

The topological entropy is related to the measure-theoretic
entropy by the so-called 
variational principle.
It states that 
\cite[Theorem~8.6]{Wa}
\index{variational principle} 
\begin{equation}
  \label{eq:var_princ}
  h_{top}(g)=\sup \{h_\mu(g): \mu\in {\mathcal M}(g,X)\}.
\end{equation}
A measure $\mu\in  \mathcal{M}(g,X)$ for which 
$ h_{top}(g)=h_\mu(g)$ is called a {\em measure of maximal entropy}.\index{measure!of maximal entropy|textbf}  

Let $\widetilde X$ be another compact metric space.  If $\mu$ is a
measure on $X$ and
the map
  $\varphi\: X\ra  \widetilde X$ is continuous, then
the 
{\em
  push-forward}\index{measure!push-forward}\index{push-forward!of measure}\index{$f_*\mu$}
$\varphi_*\mu$ of $\mu$ by $\varphi$ is the
measure given by $\varphi_*\mu(A)\coloneqq \mu(\varphi^{-1}(A))$ for all Borel
sets $A\sub  \widetilde X$. Note that if $\widetilde X=X$, then $\mu$ is $\varphi$-invariant if and only if $\varphi_*\mu=\mu$. 

Suppose $\widetilde g\: \widetilde X\ra \widetilde X$ is a continuous map,  $\mu \in \mathcal{M}(X,g)$, and $\widetilde \mu \in \mathcal{M}(\widetilde X, \widetilde g)$. 
Then the dynamical system $(\widetilde X,  \widetilde g, 
\widetilde \mu)$ is  called a 
{\em (topological) factor}\index{factor of dynamical system}
 of 
$(X,   g, \mu)$ if there exists a continuous and surjective map $\varphi\: X \ra \widetilde X$ such that $\varphi_*\mu = \widetilde \mu$ and 
$ \widetilde g\circ \varphi=\varphi\circ g$. Then we have the following commutative diagram:
 \begin{equation*}
   \label{eq:factor}
   \xymatrix{
     (X,\mu) \ar[r]^{g} \ar[d]_{\varphi} & (X,\mu) \ar[d]^{\varphi}
     \\
     (\widetilde X, \widetilde \mu)
      \ar[r]^{\widetilde g} & (\widetilde X, \widetilde \mu)\rlap{.}
   }
 \end{equation*}
In this case,  
$h_{\widetilde \mu}(\widetilde g)\le h_\mu(g)$ 
\cite[Proposition~4.3.16]{KH}.  

If  $\mu$ and $\nu$ are (Borel) probability measures on a compact metric space $X$, then $\mu$ has a unique Lebesgue decomposition
with respect to $\nu$. More precisely, $\mu$ can uniquely be written as  
$\mu=\mu_a+\mu_s$,  where $\mu_a$ and $\mu_s$ are finite  measures on $X$   such that    
$\mu_a$ is absolutely continuous and $\mu_s$ is singular with respect
to $\nu$ (see \cite[Theorem 6.10]{Ru}). 

We require the following fact.

\begin{lemma} \label{lem:invdecomp} Let $X$ be a compact metric space, and $g\: X\ra X$ be a continuous map. Suppose $\mu$ and $\nu$ are probability measures on $X$, and $\mu=\mu_a+\mu_s$ is the Lebesgue decomposition of $\mu$ with respect to $\nu$. If $\mu$ and $\nu$ are $g$-invariant, then $\mu_a$ and $\mu_s$ are also $g$-invariant.  
\end{lemma}

\begin{proof} 
  We claim that the measure $g_*\mu_a$ is absolutely continuous
  with respect to $\nu$. 
For this it suffices 
to show that if
  $A\sub X$ is a Borel set with $\nu(A)=0$, then
  $g_*\mu_a(A)=0$. Now $\nu$ is $g$-invariant and so
  $\nu(g^{-1}(A))=\nu(A)=0$ for such a set $A$. Since $\mu_a$ is
  absolutely continuous with respect to $\nu$, this implies that
  $g_*\mu_a(A)=\mu_a(g^{-1}(A))=0$ as desired.

Similarly, we claim that $g_*\mu_s$ is singular with respect to $\nu$. 
Since $\mu_s$ is singular with respect to $\nu$,  there exists 
  a Borel set $B\sub X$ with $\nu (B)=0$ 
and $\mu_s(X\setminus B)=0$. Then
$\mu_a(B)=0$ and $\mu_a(g^{-1}(B))=g_*\mu_a(B)=0$. This combined with the  $g$-invariance of $\mu$ implies 
that 
$$ 
\mu_s(g^{-1}(B))=\mu(g^{-1}(B))=\mu(B)=\mu_s(B). $$ 
It follows that 
\begin{align*} 
g_*\mu_s(X\setminus B) & =\mu_s(X\setminus g^{-1}(B)) =
\mu_s(X)-\mu_s(g^{-1}(B))\\
&= \mu_s(X)-\mu_s(B)= \mu_s(X\setminus B) =0.  
\end{align*} 
This shows that  $B$ is a set of full $g_*\mu_s$-measure and has $\nu$-measure zero. We conclude  that  $g_*\mu_s$ is indeed singular with respect to 
$\nu$.

Now $\mu =g_*\mu=g_*\mu_a+g_*\mu_s$. By what we have seen, here 
 $g_*\mu_a$ is absolutely continuous and $g_*\mu_s$ singular with respect to $\nu$. The uniqueness of the Lebesgue decomposition of $\mu$  implies that $g_*\mu_a=\mu_a$ and $g_*\mu_s=\mu_s$.
The statement follows. 
 \end{proof} 

\section[Construction of the measure of maximal entropy]{Construction  of the measure of maximal
  entropy} 
\label{sec:meamaxemt}

In this section we fix an expanding Thurston map $f\: S^2\ra S^2$.  Our goal is to describe a measure of maximal entropy for $f$ and show its uniqueness. We will freely use the notation and the results discussed in the previous section. 

We pick  a base metric $d$ on $S^2$ that induces the given topology. Unless otherwise indicated, metric concepts are for this metric  $d$. By Theorem~\ref{thm:main} we can find   a sufficiently high iterate $F=f^n$ of $f$ that has an  $F$-invariant 
Jordan curve $\CC\sub S^2$ with $\post(f)=\post(F)\sub \CC$. Then $F$ is also an expanding Thurston map (Lemma~\ref{lem:Thiterates}). 
In the following,  we  consider the cell decompositions $\DD^k=\DD^k(F,\CC)$ for $k\in \N_0$. 
A {\em cell} is a cell in any of the cell decompositions 
$\DD^k$, $k\in \N_0$, and the terms {\em tiles}, {\em edges}, and {\em vertices} are used in a similar way. As usual, we  denote by ${\bf X}^k$
 and ${\bf E}^k$ the set of $k$-tiles and $k$-edges  for $(F,\CC)$, respectively. 
By Proposition~\ref{prop:invmarkov} the cell decomposition   $\DD^{m+k}$ is a refinement of $\DD^k$ for $m,k\in \N_0$, 
and so  cells are subdivided by cells of higher levels.

 We denote  by
$\XOw$ and $\XOb$  the two $0$-tiles, and color the tiles for $(F,\CC)$ as in Lemma~\ref{lem:colortiles}. In particular, $\XOw$ is colored white and  
$\XOb$ is colored black. 

 For $k\in \N_0$ let  $w_k$  be the number of white and $b_k$ be  the
 number of black $k$-tiles  contained in $\XOw$, and similarly let
 $w'_k$ and $b'_k $ be the number of white and black $k$-tiles
 contained in $\XOb$. Then it follows from the discussion after 
 Lemma~\ref{lem:colortiles} that  
  \begin{equation}\label{eq:wbdeg}
   w_k+w'_k= b_k+b'_k=\deg(F)^k.
   \end{equation}
   
   Note that $ b_1, w'_1\ne 0$.  Indeed, suppose that $b_1=0$, for
   example. Then $\XOw$ contains only white $1$-tiles. Let $X\sub \XOw$
   be such a $1$-tile, $e\sub X$ be a $1$-edge with $e\sub \partial
   X$, and $Y$ be the other $1$-tile containing $e$. Then $Y$ is black
   and so $Y\sub \XOb$. Hence  
$$e\sub X\cap Y\sub \XOw\cap \XOb= \partial \XOw. $$
 Since $\partial X$ is a union of $1$-edges, it follows that $\partial
 X\sub \partial \XOw$. As $\XOw$ and $X$ are Jordan regions and $X\sub
 \XOw$, this is only possible if $X=\XOw$.  
Hence $\XOw$ is a $1$-tile and $F|\XOw$ is a homeomorphism of $\XOw$
onto itself. Applying Lemma~\ref{adhoc}~\ref{item:adhoc1} repeatedly, we see 
that $\XOw$ is a 
$k$-tile for each $k\in \N_0$. This is impossible, because $F$ is
expanding and so the diameters of $k$-tiles approach $0$  as $k\to
\infty$.   

    Define 
\begin{equation}\label{def:wb}
  w\coloneqq \frac{b_1}{b_1+w'_1}, \; b\coloneqq \frac{w'_1}{b_1+w'_1}.
\end{equation}
Then $w,b>0$ and $w+b=1$. 
It follows from \eqref{eq:wbdeg} for $k=1$ that the matrix 
\begin{equation}\label{eq:matrixA}
A=\left(\begin{array}{cc} w_{1} & b_{1}\\
w'_{1} & b'_{1}\end{array}\right)
\end{equation} 
 has the eigenvalues $\lambda_1=\deg(F)$ and $\lambda_2=w_1-b_1$ with respective  eigenvectors
 $$v_1=
 \left(\!\!\begin{array}{c} w\\
             b\end{array}\!\!\right) \text{ and } v_2=
         \left(\!\!\!\begin{array}{cc} &1\\
                       -\!\!\!\!\!&1\end{array}\!\!\right). $$
 Here $|\lambda_2|=|w_1-b_1|<\lambda_1=\deg(F)$. Indeed, since $1\le b_1\le \deg(F)$ and $0\le w_1\le \deg(F)$, we otherwise have $w_1=0$ and $b_1=\deg(F)\ge 2$. Then the white $0$-tile contains only black $1$-tiles. Arguing as in the discussion above, we see that then there can be only one such tile, and so $b_1=1$.  This is a contradiction.

The existence of a largest  positive eigenvalue $\lambda_1$ for $A$ 
with  a corresponding eigenvector  with all positive 
coordinates is an instance of the  Perron-Frobenius theorem 
(\cite[Theorem~1.9.11]{KH}). 

Let $k,l,m\in \N_0$ with $m\ge l\ge k$ be arbitrary. The map $F^k$ preserves colors  of tiles, i.e., if  $X^{m}$ is an $m$-tile, then $F^k(X^{m})$ is an $(m-k)$-tile with the same color as $X^{m}$.  Moreover, if $Y^l$ is an $l$-tile, then it follows from  Lemma~\ref{adhoc}~\ref{item:adhoc1} that the map $F^k|Y^l$  induces a bijection $X^{m}\mapsto F^{k}(X^{m})$  between the $m$-tiles contained in $Y^l$ and the $(m-k)$-tiles contained in the 
$(l-k)$-tile 
 $Y^{l-k}\coloneqq F^k(Y^l)$.

If we use this for $m=k+1$ and $l=k$, then we see   that a white $k$-tile contains 
$w_1$ white and $b_1$ black $(k+1)$-tiles, and similarly each black 
$k$-tile contains $w'_1$ white and $b'_1$ black $(k+1)$-tiles.  
This leads to the identity
\begin{equation}\label{matrixrec}   \left(\begin{array}{cc} w_{k+1} & b_{k+1}\\
w'_{k+1} & b'_{k+1}
\end{array} \right)=
\left(\begin{array}{cc} w_{k} & b_{k}\\
w'_{k} & b'_{k}
\end{array} \right) \left(\begin{array}{cc} w_{1} & b_{1}\\
w'_{1} & b'_{1}
\end{array} \right), 
\end{equation} for $k\in \N_0$. This implies  
$$ A^k= \left(\begin{array}{cc} w_{k} & b_{k}\\
w'_{k} & b'_{k}
\end{array} \right)
$$ for $k\in \N_0$. 
The following lemma is another  consequence of \eqref{matrixrec}.

\begin{lemma} \label{lem:counttiles} For all $k\in \N_0$ we have  
\begin{align*} 
  w_k&=w\deg(F)^k+b(w_1-b_1)^k, 
  &b_k&=  w\deg(F)^k-w(w_1-b_1)^k, 
  \\
  w'_k&=b\deg(F)^k-b(w_1-b_1)^k,
  &b'_k&=  b\deg(F)^k+w(w_1-b_1)^k.
\end{align*} 
\end{lemma}

 Since $|w_1-b_1|<\deg(F)$, the terms with $\deg(F)^k$ in
  these identities  are the main terms for large $k$.

\begin{proof} This follows from  \eqref{eq:wbdeg}, \eqref{def:wb}, and \eqref{matrixrec} by induction. 
\end{proof}

The next lemma provides an important connection between cells for $(F,\CC)$ and general Borel sets.

  \begin{lemma}\label{lem:SgeneratesB}
 Let $\mathcal{S}$ be the set   consisting of the empty set and the interiors of all cells for $(F,\CC)$. Then $\mathcal{S}$ is a semi-algebra generating the Borel $\sigma$-algebra $\mathcal{B}$ on $S^2$.  
   \end{lemma}

 \begin{proof} We first verify conditions  (i)--(iii) of a semi-algebra for $\mathcal{S}$.
 
\smallskip   
{\em Condition} (i):  By definition of $\mathcal{S}$ we have $\emptyset\in \mathcal{S}$.
 
 \smallskip 
{\em Condition}  (ii): Let $A,B\in \mathcal{S}$. In order to show that $A\cap B\in \mathcal{S}$,  we may assume that $A=\inte(\sigma)$ and $B=\inte(\tau)$, where $\sigma$ is a $k$-cell,  $\tau$ is an $l$-cell,  and $k\ge l$. Let $p\in \inte(\tau)$ be arbitrary. Then by Lemma~\ref{lem:uniondisjint} there  exists a unique $k$-cell $c$ with $p\in \inte(c)$. Since $\DD^k$ is a refinement of $\DD^l$, there exists a unique $l$-cell $\tau'$ with $\inte(c)\sub \inte (\tau')$ 
 (see Lemma~\ref{lem:mincell}). Then $\tau$ and $\tau'$ are both $l$-cells containing the point $p$ in their interiors. This implies that  $\tau'=\tau$, and so  $\inte(c)\sub \inte(\tau)$.
 
 It follows that $\inte(\tau)$ can be written as a disjoint union of interiors of $k$-cells. This implies that  either $A\cap B=\inte(\sigma)$ or $A\cap B=\emptyset$. In both cases, $A\cap B\in \mathcal{S}$. 
 
 \smallskip 
{\em Condition} (iii):  Let $A\in \mathcal{S}$ be arbitrary. If
$A=\emptyset$, then $S^2\setminus A=S^2$, and so $S^2\setminus A$
is equal to the 
disjoint union 
of the interiors of the $0$-cells, 
meaning it is a finite disjoint union 
of elements in  $\mathcal{S}$.

If $A=\inte(\tau)$ where $\tau$ is a $k$-cell, then $S^2\setminus
A$ is the 
disjoint union 
of the interiors of all $k$-cells distinct from $\tau$. Again
$S^2\setminus  A$ is a 
finite disjoint union 
of sets in $\mathcal{S}$.

So $\mathcal{S}$ is indeed a semi-algebra.
 
\smallskip  
{\em $\mathcal{S}$ generates $\mathcal{B}$}:  Let $\mathcal{A}$ be the smallest $\sigma$-algebra on $S^2$ containing  $\mathcal{S}$. Since $\mathcal{S}$
consists of Borel sets, we have $\mathcal{A}\sub \mathcal{B}$. So
in order to show that $\mathcal{A}=\mathcal{B}$ 
it suffices to establish 
that $U\in \mathcal{A}$ 
for each non-empty open 
subset $U$ of $S^2$. 

Let $p\in U$ be arbitrary. Then for each $k\in \N_0$ the point $p$ is contained in the interior of some  $k$-cell. Since $F$ is expanding, the diameters of $k$-cells approach $0$ as $k\to \infty$. Hence there exists a cell $c$ with $p\in \inte(c)\sub U$. This implies that $U$ is a union of elements in $\mathcal{S}$. Since for each $k\in \N_0$ there are only finitely many 
$k$-cells, the set $\mathcal{S}$ is countable, and so $U$ is a countable union of elements in $\mathcal{S}$. Hence $U\in \mathcal{A}$ as desired. 
\end{proof}  

In the following, we set 
 $$E^\infty=\bigcup_{k\in \N_0}F^{-k}(\CC). $$
 Then $E^\infty$ is a Borel set. 
 Proposition~\ref{prop:celldecomp}~\ref{item:skeletons} (applied to the map $F$) implies that   $E^\infty$ is equal to the union of all edges.  Since every vertex is contained in an edge, the set $E^\infty$ also contains all vertices.
 Moreover, we have 
 \begin{equation}\label{eq:CinftyFinv}
  F^{-1}(E^\infty)=E^\infty.
  \end{equation}
  Indeed, note that $F^{-1}(\CC) \supset \CC$
  and so 
  \begin{align*}
    F^{-1}(E^\infty)
    &= 
    F^{-1}\biggl(\bigcup_{k\in
      \N_0}F^{-k}(\CC)\biggr) 
    =
    \bigcup_{k\in\N_0}F^{-(k+1)}(\CC)\\
    &=
    \bigcup_{k\in\N}F^{-k}(\CC)= \CC \cup  \bigcup_{k\in\N}F^{-k}(\CC) =E^\infty.
  \end{align*}

 \begin{lemma}\label{lem:generator}
 Let $\mu$ be an $F$-invariant probability measure on $S^2$ with $\mu(E^\infty)=0$. Then for each  $k\in \N$ the set $\X^k$ of $k$-tiles forms  a measurable  partition of $(S^2, \mu)$. It is equivalent to the partition $\xi^k_F$ where $\xi=\X^1$. 
 Moreover, $\xi=\X^1$ is a generator for $(F,\mu)$.  
  \end{lemma}
  
  \begin{proof} Note that $\mu(E^\infty)=0$ implies that all edges are sets of $\mu$-measure zero. Since every vertex is contained in an edge, we also have $\mu(\{v\})=0$ for all vertices $v$. The  $k$-tiles cover $S^2$, and two distinct $k$-tiles have only edges or vertices, i.e., a set of $\mu$-measure zero, in common. Hence $\X^k$ is a measurable partition of $(S^2, \mu)$. 
  
  Let $X\in \X^k$ be arbitrary. Then for $i=1, \dots, k$ there exist unique $i$-tiles 
  $X^i$ with $X=X^k\sub X^{k-1}\sub \dots \sub X^1$.
  Set $Y_i=F^{i-1}(X^i)$ for $i=1, \dots, k$.  Then $Y_1, \dots, Y_{k}$ are  $1$-tiles.  We claim that
  \begin{equation}
  \label{eq:Xxikg}
  X=Y_1\cap F^{-1}(Y_2)\cap \dots \cap   F^{-(k-1)}(Y_{k}).
  \end{equation}
To see  this,  denote the right hand side in this  equation by $\widetilde X$. Then it is clear that $X\sub \widetilde X$.  We verify   $X=\widetilde X$ by  inductively showing that  for any point $x\in \widetilde X$ we have 
  $x\in X^i$ for $i=1, \dots, k$, and so $x\in X^k=X$. 
  
  Indeed, since $\widetilde X \sub Y_1=X^1$ this is clear for $i=1$. Suppose $x\in X^{i-1}$ for some $i$ with $2\le i\le k$.  To complete the inductive step, we have to show $x\in X^i$. Note that $x\in \widetilde X\sub F^{-(i-1)}(Y_{i})$ and so $F^{i-1}(x)\in Y_{i}$. The map  
  $F^{i-1}|X^{i-1}$ is a homeomorphism of  $X^{i-1}$ onto the $0$-tile $F^{i-1}(X^{i-1})$. Moreover, $x\in X^{i-1}$, $X^{i}\sub X^{i-1}$, and $F^{i-1}(x)\in Y_{i}=F^{i-1}(X^{i})$. Hence by injectivity of $F^{i-1}$ on $X^{i-1}$ we have $x\in X^{i}$ as desired.

  Equation~\eqref{eq:Xxikg} shows that   every element in $\X^k$ belongs to $\xi^k_F$. This implies 
 that the measurable partitions $\X^k$ and  $\xi^k_F$ are equivalent ($\xi^k_F$
 may contain additional sets, but they have to be of measure zero).

 To establish  that $\xi={\bf X}^1$ is a generator, let $B\sub S^2$ be an arbitrary 
 Borel set and $\eps>0$. By what we have seen, it is enough  to
 show that there exists $k\in \N$ and a union $A$ of $k$-tiles
 such that  $\mu(A \bigtriangleup B)<\eps$.
 
 By regularity of $\mu$ 
 there exist a  
compact set $K\sub B$ 
and an open set $U\sub S^2$ 
with $K\sub B\sub U$ 
and $\mu(U\setminus K)<\eps$. 
Since the diameters of  tiles approach  $0$ uniformly as their  levels become larger, we can choose $k\in \N$ so large that every  $k$-tile that meets $K$ is contained in the open neighborhood $U$ of $K$.  Define 
$$A=\bigcup\{X\in \X^k: X\cap K\ne \emptyset\}.$$ Then
$K\sub A\sub U$. This implies $A\bigtriangleup B\sub U\setminus K$, and so  
$$\mu(A\bigtriangleup B)\le \mu(U\setminus K)<\eps$$ as desired.
 The proof is complete.  
   \end{proof}
   
The last lemma  allows us to easily compute the entropy of $F$-invariant measures $\mu$ once we know that  $\mu(E^\infty)=0$.  
   The following fact will be useful for verifying this. 

 \begin{lemma}\label{lem:coveredges}
 There exists  $1\le L<\deg(F)$ such that for all   $k,m\in \N_0$ and each  $m$-edge $e$ there exists a collection $M$ of $(m+k)$-tiles with  $\#M\le CL^{k}$ such that  $e$ is contained in the interior of the set $\bigcup_{X\in M} X$. 
 Here $C$ is independent of $k$.  
   \end{lemma}
 
 The total number of $(m+k)$-tiles is $2\deg(F)^{m+k}$. So the lemma says that for large $k$, the $m$-edge $e$ can be covered by a substantially smaller number of $(m+k)$-tiles.
 
 \begin{proof} It follows from  Lemma~\ref{lem:quasiball} and Proposition~\ref{lem:expoexp}~\ref{item:expoex2} that we  can find $k_0\in \N$ such that for every $s$-tile $X$, $s\in \N_0$,  there exist two $(s+k_0)$-tiles $Y$ and $Z$, one white and one black, with 
 $Y\sub \inte(X)$ and $Z\sub \inte(X)$. 
  
Every white $s$-tile contains $w_{k_0}$ white and $b_{k_0}$ black
$(s+k_0)$-tiles, and every black  $s$-tile contains $w_{k_0}'$
white and $b_{k_0}'$ black $(s+k_0)$-tiles. 
By \eqref{eq:wbdeg} we also know  that
$w_{k_0}+w_{k_0}'=b_{k_0}+b_{k_0}'=\deg(F)^{k_0}$. 
The choice of $k_0$ ensures that $ w_{k_0},
w'_{k_0},b_{k_0},b'_{k_0}\ge 1$. By possibly choosing $k_0$
larger, we may also assume that $\deg(F)^{k_0}\ge 3$. 

 Now let $e$ be an arbitrary $m$-edge. For each 
 $l\in \N_0$ we will define certain collections $T_l$ of 
 $(m+lk_0)$-tiles  whose union contains $e$ in its interior. We  denote the number of  white tiles in  $T_l$ by $N^{\tt w}_l$, the number of  black tiles in $T_l$ by $N^{\tt b}_l$, and define  $N_l=\max\{N_l^{\tt w}, N^{\tt b}_l\}$. Then the number of tiles in $T_l$ is bounded by  $2N_l$.

  Let $T_0$ be the set of all $m$-tiles that meet $e$.  Then the union of  the tiles in 
  $T_0$ is the closure of the edge flower of $e$ and so it contains  $e$  in its interior.
    
  Suppose the collection $T_l$ has been constructed. 
  Then we   subdivide each of the tiles 
  $U$ in $T_l$  into $(m+(l+1)k_0)$-tiles and  remove one white and one black  
$(m+(l+1)k_0)$-tile 
contained in the interior of 
$U$. We define  $T_{l+1}$ as  the collection of all  tiles obtained in this way from tiles in $T_l$.  Since $\inte(U)\cap e=\emptyset$ for each $U\in T_l$, the  union of the tiles in $T_{l+1}$ still contains  $e$
in its interior. 
Then for  the number of white tiles in $T_{l+1}$ we have the estimate 
\begin{align*}
  N^{\tt w}_{l+1}&= N^{\tt w}_l(w_{k_0}-1)+N^{\tt b}_l(w_{k_0}'-1) 
  \\
  &\le  N_l (w_{k_0}+w_{k_0}'-2) =N_l(\deg(F)^{k_0}-2).  
\end{align*}  
Similarly, $$N^{\tt b}_{l+1}\le  N_l(\deg(F)^{k_0}-2), $$ and so 
$$N_{l+1}\le N_l(\deg(F)^{k_0}-2).$$ 
 
Let $$L\coloneqq (\deg(F)^{k_0}-2)^{1/k_0}.$$
Then $1\le L<\deg(F)$ and
$$\#T_l\le 2N_l\le 2N_0 L^{k_0l}$$ 
is a bound for  the total number of tiles in $T_l$. 

Now let $k\in \N_0$ be arbitrary. Then we can choose  $l\in \N_0$ such that $k
\leq lk_0 < k+ k_0$. For each  
$(m+lk_0)$-tile $U$ in $T_l$ we can pick an $(m+k)$-tile that contains $U$. Let $M$ be that collection of all $(m+k)$-tiles obtained in this way. Then the union of all tiles in $M$ contains  $e$ in its interior and we have 
$$ \#M\le \#T_l\le 2N_0 L^{k_0l}\le 
2N_0 L^{k_0}L^{k}=CL^k,$$
where $C=2N_0L^{k_0}$. The claim follows. 
    \end{proof}
 
 The constant $C$ in the previous lemma depends on $e$.
 If we require the weaker property that the collection $M$ of $(m+k)$-tiles  only covers $e$, then we can choose the collection so 
 that $\#M\le CL^k$ with a constant $C$ independent of $e$. Indeed, in this case,  we can choose $T_0$ to consist of the two $m$-tiles $X$ and $Y$, one white and one black, that contain $e$ in their  boundary. Then $N_0=1$ and this leads to an inequality of the desired type with a constant $C$ independent of $e$.  
 
 In the next lemma we obtain an upper bound for the  topological entropy of $f$.
 
 \begin{lemma}  \label{lem:topent} $h_{top}(f)\le\log (\deg(f))$. 
 \end{lemma}
 We will verify later that actually  $h_{top}(f)=\log (\deg(f))$ (see the proof of Corollary~\ref{cor:topent}).

 \begin{proof} Since $h_{top}(F)=nh_{top}(f)$ and $\deg(F)=\deg(f)^n$, it suffices to show that $h_{top}(F)\le \log(\deg(F))$.

 To see  that 
$h_{top}(F)\le \log(\deg(F))$,  let $\eps>0$ be arbitrary. 
Since $F$ is expanding, we can find $k_0\in \N_0$ such that 
$\diam(X)\le \eps$ whenever $X\in \X^k$ for $k> k_0$. 

Now if $k\in \N$ and  $X\in \X^{k+k_0}$ are  arbitrary, then 
$F^i(X)$ is a tile of level $k-i+k_0>k_0$ for $i=0,1,\dots, k-1$, and so 
$\diam(F^i(X))\le \eps$. This implies that the diameter of $X$ 
with respect to the metric $d^k_F$ derived from our base metric $d$  is $\le \eps$ (see \eqref{eq:defdng}). 
Since the number of $(k+k_0)$-tiles is equal to $2\deg(F)^{k+k_0}$ and these tiles form a cover of $S^2$, it follows that 
$D(F,\eps,k)\le 2\deg(F)^{k+k_0}$, and so 
$ h( F, \eps)\le \log(\deg(F))$. Letting $\eps\to 0$ we conclude 
$h_{top}(F)\le \log(\deg(F))$ as desired. 
\end{proof}
 
 Since the curve $\CC$ is $F$-invariant, we can restrict $F$ to $\CC$ to obtain  a map $F|\CC\: \CC\ra \CC$. The following lemma shows that the topological entropy of this restriction is strictly smaller than  $\log (\deg(F))$.
 
  \begin{lemma}  \label{lem:topents} $h_{top}(F|\CC)<\log (\deg(F))$. 
 \end{lemma}
 
 \begin{proof} 
The proof is very similar to the proof
 of Lem\-ma~\ref{lem:topent}.  Again let   $d$  be a base metric on $S^2$.  
 
 Since $\CC$ consists of finitely many  $0$-edges, by Lemma~\ref{lem:coveredges} we can cover $\CC$ by a collection $M_k$ of 
$k$-tiles, where $\#M_k\le CL^k$. Here $1\le L<\deg(F)$ and $C$ is independent of $k$. The $k$-edges in the boundaries  
of the $k$-tiles in $M_k$ then form a cover of $\CC$. 
It is clear that each $k$-edge contained in $\CC$ must belong to this collection. 
Hence if $E_k$ is the set of all $k$-edges contained in $\CC$, we have $\#E_k\le C'L^k$ with a constant $C'$ independent of $k$.

Now  let $\eps>0$ be arbitrary. 
Since $F$ is expanding, we can find $k_0\in \N_0$ such that 
$\diam(X)\le \eps$ whenever $X\in \X^k$ for $k> k_0$. 
Since every $k$-edge is contained in a $k$-tile, we also have $\diam(e)\le \eps$ whenever $e\in \E^k$ for $k> k_0$. 

If  $k\in \N$ and  $e\in E_{k+k_0}$ are  arbitrary, then 
$F^i(e)$ is an edge of level  $k-i+k_0>k_0$ 
for $i=0,1,\dots, k-1$, 
and so 
$\diam(F^i(e))\le \eps$. This implies that the diameter of $e$ 
with respect to the metric $d^k_F$  is $\le \eps$.

 It follows that 
$D(F|\CC,\eps,k)\le \#E_{k+k_0}\le C'L^{k_0+k}$, and therefore 
$ h(F|\CC, \eps)\le \log(L)$. Letting $\eps\to 0$ we conclude 
$h_{top}(F|\CC)\le \log(L)<\log(\deg(F))$ as desired. 
\end{proof} 

\begin{rem}\label{rem:topentFC} The topological entropy of $F|\CC$ can in fact be computed
explicitly. Since we do not need this in the following, we will only give an outline of the procedure without  proof.

 As described in Section~\ref{sec:labelings}, we can  label the $0$-edges  $e_1, \dots, e_{m}$ so that they are in cyclic order  on the boundary of the white $0$-tile
   $X^0_{\wt}$. Here $m=\#\post(F)$.  If $e^n$ is an $n$-edge, then $F^n$ maps it homeomorphically to a $0$-edge. We say that   $e^n$ is  of {\em type} $i$ if
   $F^n(e^n)= e_i$. 

 Each $0$-edge is subdivided into  $1$-edges. 
   Let $a_{ij}$ be the number of $1$-edges of type $j$ into which
   $e_i$ is subdivided, and  $A= (a_{ij})_{i,j=1, \dots, m}$ be the  
   $(m\times m)$-matrix with entries $a_{ij}$. It is easy to see that for $n\in \N$  the entry $a^n_{ij}$ of the $n$-th
   power $A^n=(a^n_{ij})$ of $A$ gives the number of $n$-edges of type 
   $j$ into which the $0$-edge $e_i$ is subdivided. Based on this one can show that for  the topological entropy of $F|\CC$ we have 
   $h_{top}(F|\CC)= \log(\rho(A))$, where $\rho(A)$ is the spectral radius of $A$.
   
   The spectral radius  $\rho(A)$ in turn  can easily be computed from any matrix norm. If for an $(m\times m)$-matrix $B=(b_{ij})$ we set 
   $$ \norm{B}\coloneqq \sum_{i,j=1}^m\abs{b_{ij}}$$
   for example, then $\displaystyle
   h_{top}(F|\CC)= \rho(A)=\lim_{n\to \infty} \norm{A^n}^{1/n}.$
   Note that for  this matrix norm, $\norm{A^n}$ is equal to the number 
   of $n$-edges contained in $\CC$.

   Recall from Chapter~\ref{cha:subdivisions} that the
   $F$-invariant Jordan curve $\CC$ gives a two-tile subdivision
   rule that is realized by $F$ (see
   Proposition~\ref{prop:ThmapSub}). The quantity
   $h_{top}(F|\CC)$ can be viewed as a measure of the complexity
   of this two-tile subdivision rule (see also \cite{MT88}). This
   is similar to the ``core-entropy'' of the Hubbard tree of a
   postcritically-finite polynomial (see \cite{Ti16}).
 \end{rem}

After these preparations we are now ready to construct a measure $\nu_F$ that will turn out to be the unique measure of maximal entropy for $F$ and $f$. 
 
\begin{prop}
  \label{prop:exmeasure} 
  \index{measure!of maximal entropy}
  \index{n@$\nu_f$} 
  There exists a unique probability measure $\nu_F$ on $S^2$  such that for each $X\in \X^k$, $k\in \N_0$, we have \begin{equation}\label{tilemass}
\nu_F(X)=
\left\{
\begin{array} {c}
 w \deg (F)^{-k} \\  b \deg (F)^{-k}
 \end{array}\right. \text {if}
 \begin{array} {c} \text{  $X$ is  white,}\\
\text{ $X$ is black.}
\end{array} 
\end{equation} 

Then  $\nu_F(E^\infty)=0$. Moreover, the measure $\nu_F$ is  $F$-invariant,   $F$ is mixing  for $\nu_F$, and
$h_{\nu_F}(F)=\log(\deg(F))$. 
\end{prop}

Here $w$ and $b$ are as in \eqref{def:wb}. The proposition implies in
particular that edges and vertices are sets of $\nu_F$-measure zero,
and that $F$ is ergodic for $\nu_F$. 

\begin{proof} The proof proceeds in several steps.

\smallskip  {\em Construction of $\nu_F$ and $\nu_F(E^\infty)=0$}:
For each $k$-tile $X$, $k\in \N_0$,  set  $w(X)=w (\deg (F))^{-k}$ if $X$ is white and $w(X)=b (\deg (F))^{-k}$ if $X$ is black.

If $X\in \X^k$ is white, then  $w_1$ is  the number of white $(k+1)$-tiles contained in $X$, and $b_1$ the number of black $(k+1)$-tiles contained in $X$. Since $w_1+w'_1=\deg( F)$, we have 
  \begin{align*}
    \sum_{Y\in \X^{k+1}, Y\sub X} w(Y)&=
    \frac{w_1w + b_1 b}{\deg (F)^{k+1}}=\frac{w_1b_1 + b_1 w'_1}{(b_1 + w'_1)\deg (F)^{k+1}}
    \\
    &=\frac{b_1}{(b_1 +w'_1)\deg( F)^k}
    =w(X).
  \end{align*}
A similar equation is also true for black $k$-tiles. If we iterate these identities, we  get 
\begin{equation}\label {martingale}
 \sum_{Y\in \X^{k+m}, Y\sub X} w(Y)=w(X)
 \end{equation} 
for all $k,m\in \N_0$ and all $X\in \X^k$. 
 
For  $A\sub S^2$ we now define 
\begin{equation}\label{defmu*}
\nu^*(A)=\inf_{{\mathcal U}} \sum_{X\in {\mathcal U}} w(X),
\end{equation}
where the infimum is taken over all  covers ${\mathcal U}$ of $A$ by tiles (not necessarily of the same level). By  subdividing   tiles  into tiles of high level and using  \eqref{martingale}, one sees 
that  in the infimum in the definition of $\nu^*(A)$ it is enough
to only consider   covers by tiles whose levels exceed a given number $k$. Based on  this and the fact that $\max_{x\in \X^k}\diam(X)\to 0$ as $k\to \infty$, it is clear  that $\nu^*$ is a {\em metric outer measure}, i.e., $\nu^*$ is an outer measure and  if $A,B\sub S^2$ are sets with $\dist(A,B)>0$,  then 
$$\nu^*(A\cup B)=\nu^*(A)+\nu^*(B). $$
It is a known fact that the restriction of a metric outer measure to the $\sigma$-algebra of Borel sets is  a measure (see \cite[Theorem~1.2, p.~267]{StS}). We denote this restriction of $\nu^*$ by $\nu_F$. 

If $A\sub S^2$ is compact, then it is enough to consider only finite covers by tiles in \eqref{defmu*}. Indeed, suppose 
${\mathcal U}=\{X_i: i\in \N\}$ is an infinite cover of the compact set $A\sub S^2$ by tiles.  Let $\eps>0$ and 
$i\in \N$  be arbitrary. By covering the edges on the boundary of $X_i$ by tiles of high level  as in Lemma~\ref{lem:coveredges}  we  can find 
a finite collection ${\mathcal U}_i$ of tiles (including $X_i$) such that 
$ X_i\sub \inte(X'_i), $ where 
$$   X'_i=\bigcup_{X\in {\mathcal U}_i}X$$ and 
$$\sum_{X\in  {\mathcal U}_i}w(X)\le w(X_i)+\eps/2^i. $$
Finitely many  sets 
$\inte(X'_{i_1}), \dots, \inte(X'_{i_m})$ will cover 
$A$.  Then 
$$\mathcal{U}'= \mathcal{U}_{i_1}\cup \dots \cup
\mathcal{U}_{i_m}$$
is a finite collection of tiles that covers $A$, and we have
$$\sum_{X\in  {\mathcal U}'}w(X)\le 
\sum_{X\in  {\mathcal U}}w(X)+\eps. $$
Since $\eps>0$ was arbitrary, we conclude that for compact sets $A$ we get the same infimum in \eqref{defmu*} if we restrict ourselves to finite covers by tiles. 

A  consequence of this is that  $\nu_F(X)=w(X)$ for each $X\in \X^k$, $k\in \N_0$. Indeed, by definition of $\nu_F$ we obviously have $\nu_F(X)\le w(X)$. 

 For an inequality in the opposite direction, it is enough to  consider an arbitrary finite cover ${\mathcal U}$     of $X$ by tiles. By subdividing the tiles in  ${\mathcal U}$ if  necessary, we may also assume that they all have the same level $l$ and that $l\ge k$. Since ${\mathcal U}$ is a cover of $X$ and $l$-tiles have pairwise disjoint interiors, 
 this implies that $Y\in  {\mathcal U}$ whenever $Y\in \X^l$ and $Y\sub X$. Hence 
 $$w(X)=\sum_ {Y\in \X^{l}, Y\sub X} w(Y)\le  \sum_ {Y\in {\mathcal U}} w(Y). $$
 Taking the infimum over all ${\mathcal U}$ we get
 $  w(X)\le \nu_F(X)$ as desired. 

Since $\nu_F(X)=w(X)$ for all tiles $X$, we have  \eqref{tilemass}. 

It  follows from Lemma~\ref{lem:coveredges} and the definition of $\nu_F$, that if $e$ is an edge, then 
$\nu_F(e)=0$. Since $E^\infty$ is the (countable) union  of all edges, we have $\nu_F(E^\infty)=0$.  This also shows that 
$$\nu_F(S^2)=\sum_{X\in \X^0} \nu_F(X)=\sum_{X\in \X^0} w(X)=w+b=1, $$
and so $\nu_F$ is a probability measure.

\smallskip 
{\em Uniqueness of $\nu_F$}:
Suppose that $\nu$ is another  probability measure on $S^2$ satisfying  the analog of \eqref{tilemass}. Then from Lemma~\ref{lem:coveredges} it follows that each edge is a  set of $\nu$-measure zero.   Hence  $\nu(\inte(c))= \nu_F(\inte(c))$ whenever  $c$ is a cell for $(F,\CC)$. Since the empty set together with the  interiors of all cells for $(F,\CC)$ form a semi-algebra $\mathcal{S}$  generating the Borel $\sigma$-algebra on $S^2$ (see Lemma~\ref{lem:SgeneratesB}), we conclude that $\nu=\nu_F$.

\smallskip 
{\em  $\nu_F$ is $F$-invariant}:
To show that $\nu_F$ is $F$-invariant, it is enough to  verify that 
 \begin{equation}\label{eq:muFinvF}
 \nu_F(F^{-1}(A))=\nu_F(A)\end{equation}   for all sets $A$ in the semi-algebra $\mathcal{S}$.   This is true if $A=\emptyset$. 
 
 Edges  are sets of $\nu_F$-measure zero, and the preimage  of an edge is a finite union of edges (see Proposition~\ref{prop:celldecomp}~\ref{item:fkunioncells}). This implies that \eqref{eq:muFinvF} holds  if $A=\inte(e)$  for an  edge $e$,  or  if $A=\{v\}$ for a vertex  $v$ (since  every vertex is contained in an edge).
 
 Moreover, if $X$ is a tile, then $X\setminus \inte(X)$ is a union of edges, and so  we have  $\nu_F(F^{-1}(\inte(X)))=\nu_F(F^{-1}(X))$ and $\nu_F(\inte(X))=\nu_F(X)$. 
So in order  to establish \eqref{eq:muFinvF},
 it remains  to show that  
 $$\nu_F(F^{-1}(X))= \nu_F(X)$$ for all tiles $X$. 
To see this, note that 
if $X$ is a $k$-tile, then $F^{-1}(X)$ is a union of $\deg(F)$ $(k+1)$-tiles that have the same color as $X$.  Since the intersection of any two distinct  $(k+1)$-tiles is contained in a union of edges, and hence a set of $\nu_F$-measure zero,  it follows from \eqref{tilemass} that 
 \begin{equation*}
    \nu_F(F^{-1}(X))=\deg(F) \frac{\nu_F(X)}{\deg( F)}=\nu_F(X).
  \end{equation*}
  
  \smallskip 
{\em $F$ is mixing for  $\nu_F$}: It suffices to show that for all sets $A$ and $B$ in the semi-algebra $\mathcal{S}$ we have 
$$\nu_F(F^{-m}(A)\cap B)\to  \nu_F(A)\nu_F(B)$$ 
as $m\to \infty$. Based on the fact that edges are sets of $\nu_F$-measure zero and that the preimage of each edge under $F$ is a finite union of edges, for this it suffices to show that  for all tiles $X$ and $Y$ we  have 
$$\nu_F(F^{-m}(X)\cap Y)\to  \nu_F(X)\nu_F(Y)$$ 
as $m\to \infty$. 

 So let $k,l,m\in \N_0$,   $X=X^k\in \X^k$ and $Y=Y^l\in \X^l$  be arbitrary. We may assume that $m\ge l$. 
   Then $F^{-m}(X^k)$ is a union of $(m+k)$-tiles that have the same color as $X^k$. Since edges are sets of $\nu_F$-measure zero and the $(m+k)$-tiles subdivide the $l$-tile $Y^l$, it follows that 
    \begin{equation}
    \label{book1} \nu_F(F^{-m}(X^k)\cap Y^l)=\\ \frac{\nu_F(X^k)}{\deg(F)^m}\cdot \#M, 
    \end{equation}
    where 
    \begin{equation} \label{eq:defsetM} 
    M\coloneqq 
    \{Z^{m+k}\in \X^{m+k}: Z^{m+k}\sub Y^l, \, F^m(Z^{m+k})=X^k\}. 
 \end{equation}

    Let $X^0\in \X^0$ be the unique $0$-tile with $X^k\sub X^0$,
    and $Y^0\coloneqq F^l(Y^l)\in \X^0$.  We assume that $X^0$
    and $Y^0$ are both white; the other cases are similar. Then
    $Y=Y^l$ is also white.

    \begin{figure}
      \centering
      \begin{overpic}
        [width=8cm, tics=10,
        ]{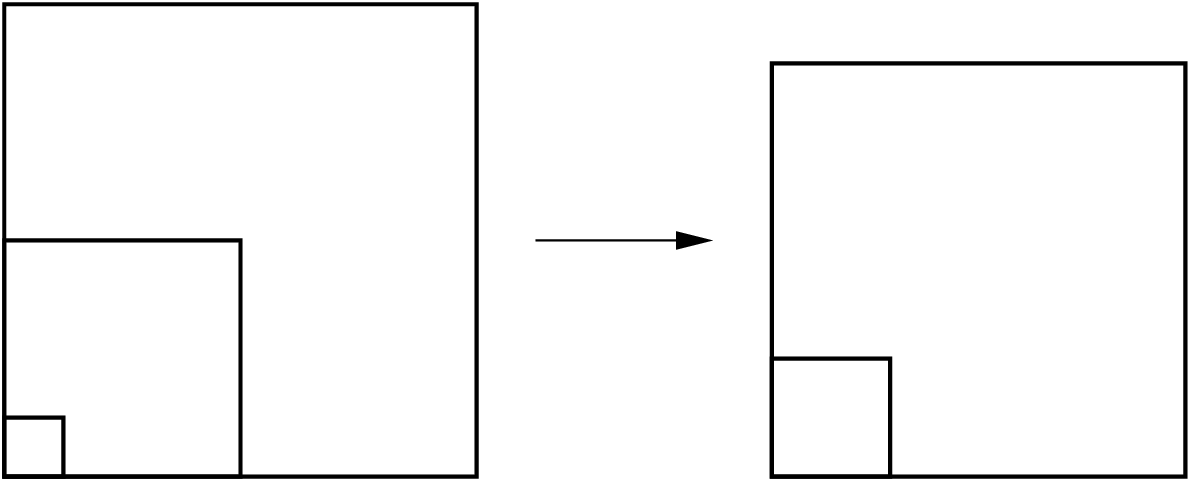}
        \put(0,-4){${\scriptstyle Z^{m+k}}$}
        \put(10,10){${\scriptstyle Z^{m}}$}
        \put(26,26){${\scriptstyle Y^{l}}$}
        \put(51,22){${\scriptstyle F^{m}}$}
        \put(67,4){${\scriptstyle X^{k}}$}
        \put(82,18){${\scriptstyle X^{0}}$}
      \end{overpic}
      \caption{Bijection of tiles.}
      \label{fig:bijection}
    \end{figure}
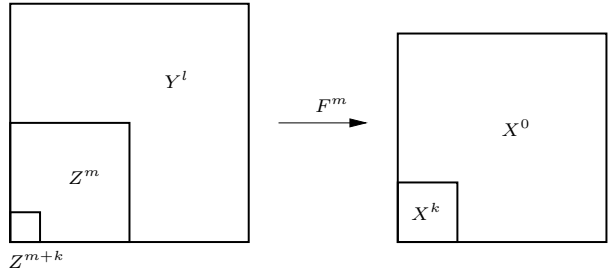
  
  Every $(m+k)$-tile $Z^{m+k}$ lies in a unique ``parent'' $m$-tile $Z^m$. Since $Y^l$ is an $l$-tile and $m\ge l$, we have $Z^{m+k}\sub Y^l$ if and only if $Z^m\sub Y^l$.  If $F^m(Z^{m+k})=X^k$, then $F^m(Z^m)$ is a $0$-tile containing $X^k$, and so 
  $F^m(Z^m)=X^0$. 
Conversely,   if $Z^m$ is an $m$-tile and $F^m(Z^m)=X^0$, then it follows from 
 Lemma~\ref{adhoc}~\ref{item:adhoc1} that $Z^{m+k}\coloneqq (F^m|Z^m)^{-1}(X^k)$ is the
 unique $(m+k)$-tile with $Z^{m+k}\sub Z^m$ and
 $F^m(Z^{m+k})=X^k$. The situation is illustrated in
 Figure~\ref{fig:bijection}. 
These statements  imply that the map 
  $Z^{m+k}\mapsto Z^m$ that assigns to each $(m+k)$-tile $Z^{m+k}$ its unique parent $m$-tile $Z^m$  induces a bijection between the set
  $M$ defined in  \eqref{eq:defsetM}  
 and 
  $$N\coloneqq \{Z^m\in \X^m: Z^m\sub Y^l,\, F^m(Z^m)=X^0\}. $$  
  Hence $\#M=\#N$.

%
  
  Since $X^0$ is white, $\#N$  is equal to the number
  of white $m$-tiles contained in $Y^l$. Applying the
  homeomorphism $F^l|Y^l$, we see that this number is equal to
  $w_{m-l}$, the number of white $(m-l)$-tiles contained in the
  white $0$-tile $Y^0=F^l(Y^l)$. It follows that $\#M=\#N= w_{m-l}$.
  So from 
  \eqref{book1} and Lemma~\ref{lem:counttiles}
  we conclude that 
  $$\nu_F(F^{-m}(X)\cap Y)=\\ \frac{\nu_F(X)}{\deg(F)^m}\cdot w_{m-l}\to  \frac{\nu_F(X)}{\deg(F)^l}\cdot w=\nu_F(X)\nu_F(Y)$$ as $m\to \infty$. 
   
  \smallskip
  {\em The identity $h_{\nu_F}(F)=\log(\deg(F))$}:  
According to Lemma~\ref{lem:generator}, 
the measurable partition $\xi=\X^1$ is a generator for $(F,\nu_F)$, and so  $h_{\nu_F}(F)=h_{\nu_F}(F,\xi)$. Moreover, for each $k\in \N$ the measurable partition $\xi_F^k$ is equivalent to the measurable partition $\X^k$ given by the $k$-tiles. 
Since the number of black and the number of white $k$-tiles are both equal to $\deg(F)^k$, it follows that   
\begin{align*}
  H_{\nu_F}(\xi^k_F)&=H_{\nu_F}(\X^k)
  =\sum_{X\in \X^k} \nu_F(X)\log(1/ \nu_F(X))
  \\
  &= w\log(\deg(F)^k/w)+b\log(\deg(F)^k/b)
  \\
  &= k\log(\deg(F))+ w\log(1/w)+b\log(1/b).  
\end{align*}

\pagebreak\noindent
This  implies  
$$h_{\nu_F}(F)=h_{\nu_F}(F,\xi)=\lim_{k\to \infty} \frac{1}{k} H_{\nu_F}(\xi^k_F)=\log(\deg(F)). 
$$
The proof is complete.
 \end{proof}

We can now identify the topological entropy of $f$.

\begin{proof}[Proof of Corollary~\ref{cor:topent}]  
  Lemma~\ref{lem:topent} (applied to $F$) implies  that 
  $h_{top}(F)$ $\le \log(\deg(F))$.   
  We also  have
  $h_{\nu_F}(F)=\log(\deg(F))$ by
  Proposition~\ref{prop:exmeasure}, and so
  $h_{top}(F)\ge \log(\deg(F))$ by the variational principle
  \eqref{eq:var_princ}. It follows that
  $h_{top}(F)= \log(\deg(F))$.  Since $F=f^n$ and so
  $\deg(F)=\deg(f)^n$ and $h_{top}(F)=n h_{top}(f)$, the claim
  follows.
\end{proof}
 
We know that $h_{\nu_F}(F)=\log(\deg(F))=h_{top}(F)$. So $\nu_F$ is a measure of maximal entropy for $F$. As we will see momentarily, it is the unique measure of maximal entropy for $F$ and also for $f$. 

\begin{theorem}  
  \label{thm:nuF}
  \index{measure!of maximal entropy}
  \index{n@$\nu_f$} 
  The measure $\nu_F$  is the unique measure of maximal entropy\index{measure!of maximal entropy}
 for $f$, i.e., the unique $f$-invariant probability measure $\nu_F$  with $h_{\nu_F}(f)=h_{top}(f)$.  Moreover, $f$ is mixing for $\nu_F$. 
\end{theorem}

\begin{proof} We first show uniqueness.  
So let $\mu$ be a probability measure that is $f$-invariant and satisfies
 $h_\mu(f)=h_{top}(f)$. Then $\mu$ is $F$-invariant and satisfies
 \begin{equation}\label{comptop}
 h_\mu(F)=nh_\mu(f)=nh_{top}(f)=h_{top}(F)=\log(\deg(F)). 
 \end{equation}
 We will see that this implies  $\mu =\nu_F$. In particular, this will show that $\nu_F$ is the unique measure of maximal entropy for $F$.  
 The proof proceeds in several steps. 
 
 By using the   Lebesgue decomposition of $\mu$ with respect to
 $\nu_F$,   we can represent    $\mu$ as a convex combination
 $\mu=\beta \mu_a+(1-\beta)\mu_s$, where $\beta\in [0,1]$,
  $\mu_a$ is a probability measure that is absolutely continuous   and $\mu_s$ is a probability measure that is singular  with respect to $\nu_F$ (if $\beta=0$  or $\beta=1$,  the decomposition is trivial and only one of the measures $\mu_a$ or $\mu_s$ exists). 
 Since $\nu_F$ and $\mu$ are $F$-invariant,  Lemma~\ref{lem:invdecomp} implies that the 
 measures $\mu_a$ and $\mu_s$ are also $F$-invariant.  Since $F$ is ergodic for $\nu_F$,  and $\mu_a$ is $F$-invariant and absolutely continuous with respect to $\nu_F$, it follows that $\mu_a=\nu_F$. 

If $\beta=1$, then $\mu=\nu_F$ and we are done.
If  $\beta\in [0,1)$, then we can use the equation 
\begin{align*}
  \log(\deg(F))&= h_\mu(F)\,=\, \beta
                   h_{\nu_F}(F)+(1-\beta)h_{\mu_s}(F)
\\
               &=
                   \beta\log(\deg(F))+ (1-\beta) h_{\mu_s}(F)
\end{align*}
to conclude that 
$$h_{\mu_s}(F)=\log(\deg(F)). $$

We will show that this is impossible by proving  that for every $F$-invariant
probability measure $\mu$ that is singular with respect to 
$\nu_F$ we must have $$h_{\mu}(F)<\log(\deg(F)). $$
The uniqueness of $\nu_F$ will then follow. 

So let  $\mu$ be  such a measure and consider the union   $E^\infty$ of all edges. Assume first that $\mu(E^\infty)>0$. By \eqref{eq:CinftyFinv}
 we can then write  $\mu$ as a convex combination 
 $\mu=\alpha \mu_1+(1-\alpha)\mu_2$ of two $F$-invariant probability measures 
 $\mu_1$ and $\mu_2$, where $\alpha=\mu(E^\infty)$,  $\mu_1$ is concentrated on $E^\infty$,  and 
 $\mu_2$ on $S^2\setminus E^\infty$ (if $\alpha=1$, this decomposition is again trivial). 
 
 Since $\mu_1$ is $F$-invariant, we have $\mu_1(F^{-k}(\CC))=\mu_1(\CC)$ for all $k\in \N_0$. On the other hand, $\CC\sub  F^{-k}(\CC)$, and so 
 $\mu_1(F^{-k}(\CC)\setminus\CC)=0$. This implies that $\mu_1(E^\infty\setminus \CC)=0$. So  $\mu_1$ is actually  concentrated on $\CC$. 
 Therefore,  by the variational principle \eqref{eq:var_princ}
 and by Lemma~\ref{lem:topents} we have  
 $$h_{\mu_1}(F)=h_{\mu_1}(F|\CC)\le h_{top}(F|\CC)<\log(\deg(F)).$$
We also have $h_{\mu_2}(F) \le h_{top}(F)= \log(\deg(F))$, and so 
 $$h_\mu(F)=\alpha h_{\mu_1}(F)+(1-\alpha)h_{\mu_2}(F)<\log(\deg(F)). $$
 In this case we are done. 
 
 In the remaining case we have  $\mu (E^\infty)=0.$ Then by 
 Lemma~\ref{lem:generator} $\xi=\X^1$ is a generator for
 $(F,\mu)$, 
and so $h_\mu(F)=h_\mu(F,\xi)$.  
In particular, 
 $$h_\mu(F)=\lim_{k\to \infty}\frac 1k \sum_{X\in \X^k} \mu(X)\log(1/\mu(X)),  $$
and the limit is bounded from above by each of the sequence elements (see \eqref{eq:linhginf}). 

Since $\mu$ and $\nu_F$ are mutually singular, we can find a Borel set $A\sub S^2$ with $\mu(A)=1$ and  $\nu_F(A)=0$. Using inner regularity of $\mu$ and  outer regularity of $\nu_F$, for each $\eps>0$ we can find a compact set $K\sub S^2$ and an open set $U\sub S^2$ with $K\sub A\sub U$, $\mu(K)>1-\eps$, and $\nu_F(U)<\eps$.
If $k$ is sufficiently large, then we can cover the set $K$ by $k$-tiles contained in $U$. 

Using this for smaller and smaller $\eps>0$, we conclude  that
for each $k\in \N$ we can find a set $M_k\sub \X^k$ 
such that for $A_k\coloneqq \bigcup_{X\in M_k}X$
we have $\mu( A_k)\to 1$ and $\nu_F(A_k)\to 0$ as 
$k\to \infty$.  Note that $\nu_F(X)\ge c\deg(F)^{-k}$ for each $X\in \X^k$, where $c>0$ is independent of $k$ and $X$. Hence 
$$\#M_k\le \nu_F(A_k)\deg(F)^k/c. $$
We also have $\#\X^k=2\deg(F)^k$.

The function $x\mapsto \phi(x)= x \log(1/x)$ is concave and has a maximum equal to $1/e$ on $[0,1]$. This implies that
if $M\sub \X^k$ is arbitrary and $A=\bigcup_{X\in M}X$, then 
\begin{align*}
  \sum_{X\in M} \mu(X)\log(1/\mu(X)) 
  &\le \#M\cdot \phi\biggl(\frac1 {\#M} \sum_{X\in M}\mu(X)\biggr)
  \\&=\mu(A) \log(\#M/\mu(A))\le \mu(A) \log(\#M)+1/e. 
\end{align*} 
To reach a contradiction, let us assume that $h_\mu(F)=\log(\deg(F))$. Then
for each $k\in \N$ we have  
\begin{align*}
  k\log(\deg(F))&=kh_\mu(F)
  \\
  &\le \sum_{X\in \X^k} \mu(X)\log(1/\mu(X))
  \\
  &= \sum_{X\in M_k}\mu(X)\log(1/\mu(X)) \, +\, \sum_{X\in
    \X^k\setminus M_k} \mu(X)\log(1/\mu(X))
  \\
  &\le \mu(A_k) \log( \#M_k)+\mu(S^2\setminus A_k) \log(\#\X^k)+C_1
  \\
  &\le \mu (A_k)\log(\nu_F(A_k)) +\,(\mu(A_k) +
  \mu(S^2\setminus A_k))\log( \deg(F)^k)+C_2
  \\
  &= \mu (A_k)\log(\nu_F(A_k))+ k\log(\deg(F)) +C_2.
\end{align*}
Here the constants $C_1$ and $C_2$ do not depend on $k$. 
An inequality of this type is  impossible as $$\mu (A_k)\log(\nu_F(A_k))\to -\infty$$ for  $k\to \infty$.  

This shows that if there is a measure of maximal entropy for $f$, then it has to agree with $\nu_F$.  We have also  proved  that $\nu_F$ is the unique measure of maximal entropy for $F$. 

We now show that  $\nu_F$ is $f$-invariant and a measure of maximal entropy for $f$.  
  Indeed, the measure $f_*\nu_F$ is $F$-invariant and  the triple $(S^2, F, f_*\nu_F)$ 
 is a factor  of $(S^2, F,\nu_F)$ by the map $f$. It follows that
 $h_{f_*\nu_F}(F)\le h_{\nu_F}(F)$. Iterating this and noting that $f^n_*{\nu_F}=F_*\nu_F=\nu_F$ by $F$-invariance of $\nu_F$, we obtain 
 $$h_{\nu_F}(F)=h_{f^n_*\nu_F}(F)\le h_{f^{n-1}_*\nu_F}(F)\le \dots\le 
  h_{f_*{\nu_F}}(F)\le h_{\nu_F}(F). $$
 Hence $h_{f_*\nu_F}(F)= h_{\nu_F}(F)$, and so $f_*\nu_F$ is a measure of maximal entropy for $F$. By uniqueness of $\nu_F$ we have $f_*{\nu_F}=\nu_F$ showing that $\nu_F$ is $f$-invariant. Moreover,
 $$  h_{\nu_F}(f)=h_{\nu_F}(F)/n=\log(\deg(F))/n=\log(\deg(f))=h_{top}(f),$$ and  so $\nu_F$ is a measure of maximal entropy for $f$. 
 By the first part of the proof we know that it is the unique such measure.

It remains to show that $f$ is mixing for  $\nu_F$. Indeed, 
since $F$ is mixing for $\nu_F$ (Proposition~\ref{prop:exmeasure}) and $\nu_F$ is $f$-invariant, we have that for all   $m \in \{0, \dots, n-1\}$, and all Borel sets 
$A,B\sub S^2$,
\begin{align*} 
  \nu_F(f^{-(nl+m)}(A)\cap B)&=\nu_F(F^{-l}(f^{-m}(A))\cap B) \\
   &\to\nu_F(f^{-m}(A))\nu_F(B)=\nu_F(A)\nu_F(B)
\end{align*} 
as 
$l \to \infty$. This implies the desired  relation
$$ \nu_F(f^{-k}(A)\cap B) \to \nu_F(A)\nu_F(B)$$
as $k\to \infty$.  The proof is complete. 
 \end{proof}
 
 \begin{proof}[Proof of Theorem~\ref{thm:maxentr0}]
   The statement follows from 
   Theo\-rem~\ref{thm:nuF}.
 \end{proof} 
 
 We have seen that $\nu_f=\nu_F$ for the special iterate $F=f^n$ that was chosen at the beginning of this section. 
 As we will now see, this identity  for the measures of maximal entropy remains true for an arbitrary iterate of $f$ (this was formulated in Corollary~\ref{cor:momeit}). 
 
 \begin{proof} [Proof of Corollary~\ref{cor:momeit}]
 Let $f\:S^2\ra S^2$ be an expanding Thurston map, $F=f^n$ be an arbitrary iterate  of $f$, and $\nu=\nu_f$ be the unique measure of maximal entropy of $f$. Then $F$ is also an expanding Thurston map (Lemma~\ref{lem:Thiterates}).
 Since $\nu=\nu_f$ is $f$-invariant, this measure  is also $F$-invariant. Moreover, by \eqref{eq:hfhF} and Corollary~\ref{cor:topent}  we have   
 $$ h_\nu(F) =n h_\nu (f)= n h_{top}(f) =n \log (\deg(f))= \log(\deg(F))= h_{top}(F). $$   
  Hence $\nu$ is a measure of maximal entropy for $F$. Since the
  measure  of maximal entropy $\nu_F$ for $F$  
is unique by Theorem~\ref{thm:maxentr0}, 
we have  $\nu = \nu_F$ as desired. 
  \end{proof}

\ifthenelse{\boolean{singlechapter}}{

%

%
%

\chapter{The geometry of the visual sphere}
\label{cha:geom-visu-sphere}

When $f\colon S^2\to S^2$ is an expanding Thurston map and
$\varrho$ a visual metric for $f$, we call the metric space
$(S^2,\varrho)$ the \defn{visual sphere}\index{visual sphere} for
$f$. Of course, $(S^2,\varrho)$ depends on the choice of the
visual metric $\varrho$, but by
Proposition~\ref{prop:visualsummary}~\ref{item:visd1d2} any two
such choices for a given Thurston map $f$ produce snowflake
equivalent metric spaces. Accordingly, we think of the visual
sphere of a Thurston map as uniquely determined up to snowflake
equivalence.  In this chapter we investigate geometric features
of the visual sphere that are invariant under such equivalences
and relate them to the dynamics of $f$.

The following statement is one of the main results here.
\begin{theorem}  
  [Properties of $f$ and its associated visual sphere]
  \label{thm:S2vsf}
  \index{visual metric}
  \index{metric!visual}
  Suppose $f\colon S^2\to S^2$ is an expanding Thurston map and $\varrho$ is a
  visual metric for $f$. 
  Then the following statements are true:
  \begin{enumerate}

  \item 
    \label{item:S2doubling}
    $(S^2,\varrho)$ is 
    {doubling}\index{doubling} 
    if and only if $f$ has {no
      periodic critical points}. 
  
  \item 
    \label{item:S2qsphere}
    $(S^2,\varrho)$ is quasisymmetrically
    equivalent to $\CDach$ if and only if $f$ is
    topologically conjugate to a {rational map}.
    \index{Thurston map!rational}
  \item 
    \label{item:S2Lattes}
    $(S^2,\varrho)$ is 
    {snowflake equivalent}\index{snowflake equivalent}
    to $\CDach$ if and only if $f$ is topologically conjugate to
    a {Latt\`{e}s map}.\index{Latt\`{e}s map}   
  \end{enumerate}
\end{theorem}
Here  it is understood that $\CDach$ is equipped with the chordal metric. 
Statement~\ref{item:S2qsphere} characterizes the visual spheres that
are quasispheres.\index{quasisphere}
As we have already discussed, it provides an interesting analog of Cannon's conjecture (see Section~\ref{sec:Cannconj}).

We know that  two expanding Thurston
maps are Thurston equivalent if and only if they are topologically
conjugate (Theorem~\ref{thm:exppromequiv}). Thus \ref{item:S2qsphere} and \ref{item:S2Lattes} can be reformulated as follows:
\begin{enumerate}
\smallskip 
\item [(ii')] \textit{
  $(S^2,\varrho)$ is quasisymmetrically
    equivalent to $\CDach$   if and only if $f$ is Thurston
  equivalent to a {rational Thurston map} with no periodic critical points. }
  
\item[(iii')] \textit{
  $(S^2,\varrho)$ is {snowflake equivalent} 
  to $\CDach$ if and only if $f$ is Thurston equivalent to a
  {Latt\`{e}s map}. }
\end{enumerate}
 For the ``if'' implication in (ii') one has to assume that
the rational map has no periodic critical points (or impose an equivalent condition);
see Example~\ref{ex:ratnotenf}.

Much more can be said in case \ref{item:S2doubling} of
the previous theorem. 

\begin{prop}
  \label{prop:Ahlforsreg} 
  Let $f\: S^2\ra S^2$ be an expanding
  Thurston map without periodic critical points, $\varrho$ be a visual
  metric for $f$ with expansion factor $\Lambda>1$, and  $\nu_f$ be 
  the measure of maximal entropy of $f$.
   Then   the metric measure space    $(S^2, \varrho, \nu_f)$ is Ahlfors 
$Q$-regular\index{Ahlfors regular}\index{measure!of maximal entropy}\index{Hausdorff!dimension}\index{metric!visual}\index{visual metric}\index{H_Q@$\mathcal{H}^Q_d$}\index{Hausdorff!measure}
   with
    $$Q\coloneqq \frac{\log(\deg(f))}{\log(\Lambda)}.$$
 In particular,     $(S^2,\varrho)$ has Hausdorff dimension $Q$ and 
    $$      0<\mathcal{H}_\varrho^Q(S^2)<\infty. $$
  \end{prop}
Here  $\mathcal{H}_\varrho^Q$ is   $Q$-dimensional Hausdorff measure 
on $(S^2,\varrho)$.
For the definition of the measure of maximal entropy $\nu_f$ see 
Chapter~\ref{cha:measure}. The statement implies that under the given assumptions we have $\mathcal{H}_\varrho^Q(M)\asymp \nu_f(M)$ for each Borel set 
$M\sub S^2$ with $C(\asymp)$ independent of $M$.

Since the Hausdorff dimension of $(S^2, \varrho)$ must be $\ge 2$, it also follows that $\Lambda\le \deg(f)^{1/2}$. Combining this with 
Theorem~\ref{thm:visexpfactors1}~\ref{item:visexpfactors2}, we obtain the upper bound 
$\Lambda_0(f)\le \deg(f)^{1/2}$ for the combinatorial expansion factor 
of $f$. We will  later see  that 
this is true for every expanding Thurston map without the additional assumption that $f$ has no periodic critical points (see Proposition~\ref{prop:macgrdn}). 

 Recall from
Proposition~\ref{prop:visualsummary}~\ref{item:vsfF} that a
metric $\varrho$ on $S^2$ is a visual metric for an expanding
Thurston map $f\colon S^2\to S^2$ if and only if it is a visual
metric for an iterate $F=f^n$ (which is also an expanding Thurston map by 
Lemma~\ref{lem:Thiterates}). Hence
Theorem~\ref{thm:S2vsf} immediately gives  the following corollary. 

\begin{cor}
  \label{cha:fvsF_lattes_rational}
  \index{Thurston map!iterate of}
  \index{iterate of Thurston map}
  \index{F f@$F=f^n$}
  Let $f\colon S^2\to S^2$ be an expanding Thurston map, and
  $F=f^n$ with  $n\in \N$  be an iterate of $f$.  Then the 
following statements are true:
  \begin{enumerate}
  \item
    \label{item:fvsF_rational}
    The map $f$ is topologically conjugate to a rational map
    if and only if $F$ is topologically conjugate to a rational
    map.  
  \item 
    The map $f$ is topologically conjugate to a Latt\`{e}s map
    if and only if $F$ is topologically conjugate to a
    Latt\`{e}s map.
  \end{enumerate}
\end{cor}


We now record some of the consequences of our results for
rational Thurston maps explicitly.   

\begin{theorem}
  \label{thm:main01}
  \index{Thurston map!rational}
  Let $f\colon\CDach \ra \CDach$
  be a  rational Thurston map with no periodic critical points. Then the following statements are true:
    \begin{enumerate}
  \item
    \label{item:frat1_qc_ex}
    \index{quasicircle}
    For each sufficiently large 
    $n\in \N$ there exists  a quasicircle $\CC\sub
    \CDach$ with $\post(f) \sub\CC$  that is $f^n$-invariant  (i.e., $f^n(\CC)\sub \CC$). 
  \item 
    \label{item:frat_invC_qc}
    \index{f-invariant@$f$-invariant!Jordan curve}
    \index{Jordan curve!f-invariant@$f$-invariant}
    Each $f^n$-invariant Jordan curve $\CC\subset \CDach$ with
    $\post(f)\subset\CC$ is a quasicircle.
  \item 
    \label{item:frat_Markov} Let 
    $\CC$ be an $f$-invariant Jordan curve $\CC\subset \CDach$ with
    $\post(f)\subset\CC$,  ${\bf E}^n$ be the set of $n$-edges,  and ${\bf X}^n$ be  
    the set of 
    $n$-tiles for $(f,\CC)$. 
Then the family of all edges $\{e: n\in \N_0, \, e\in {\bf
  E}^n\}$ consists of uniform quasiarcs, and the family of all
tiles $\{X : n\in \N_0,\, X\in
    \X^n\}$ of uniform quasidisks.
    \index{uniform!quasiarcs}
    \index{uniform!quasidisks}
    \index{quasiarc!uniform}
    \index{quasidisk!uniform}
  \end{enumerate}
\end{theorem}
Here the underlying metric is again the chordal metric on $\CDach$. 
For the concepts of uniformity used here see Section~\ref{sec:cc-quasicircle}.

From \ref{item:frat_Markov} it follows that the family 
$\{\partial X : n\in \N_0,\, X\in
    \X^n\}$ consists 
of uniform quasicircles.\index{uniform!quasicircles}\index{quasicircle!uniform}
Note that this and the statement about the arcs in \ref{item:frat_Markov} do not {\em a priori} follow from Proposition~\ref{prop:arc}, because we use different underlying metrics.
 We will prove though that for a rational Thurston map $f$ without periodic critical points the chordal metric is quasisymmetrically equivalent 
   to each visual metric for $f$ (see Lemma~\ref{lem:sigqsvis}). Once we know this,
   Theorem~\ref{thm:main01}~\ref{item:frat_Markov} can easily be deduced from Proposition~\ref{prop:arc}. 
 
A consequence of Theorem~\ref{thm:main01} is that each sufficiently high iterate of a rational  expanding Thurston
map $f\colon \CDach \to \CDach$  has a particularly nice Markov
partition, where    the tiles are
quasidisks.  

Another important property of visual spheres from the viewpoint of quasiconformal geometry is that they are linearly locally connected.
This and related properties will be discussed in 
 Section~\ref{sec:line-local-conn}.
 
 Theorem~\ref{thm:S2vsf}~\ref{item:S2doubling} and 
Proposition~\ref{prop:Ahlforsreg} will be proved in 
 Section~\ref{sec:periodic}. In Section~\ref{sec:rational-maps} we will establish 
Theorem~\ref{thm:S2vsf}~\ref{item:S2qsphere}. 
 We postpone the proof of Theorem~\ref{thm:S2vsf}~\ref{item:S2Lattes} 
to the end of Section~\ref{sec:lyapunov-exponent-mu} (this part
was essentially proved in \cite{Me09a}).  

\section{Linear local connectedness}
\label{sec:line-local-conn}

Recall (see Section~\ref{sec:QCgeom}) that a metric space $(X,d)$ is 
said to be 
{\em linearly locally connected}\index{linearly locally connected (LLC)} 
(often abbreviated as $\LLC$) if there exists a constant $C\ge 1$ such
that  
the following two conditions are satisfied:

\begin{enumerate}

\item
  [(LLC1)]
  \label{item:LLC1}
  If $p\in X$, $r>0$, and $x,y\in B_d(p,r)$, 
  then there exists a continuum $E\sub X$ with $x,y\in E$ and $E\sub B_d(p,Cr).$

\item
  [(LLC2)]
  \label{item:LLC2}
  If $p\in X$, $r>0$, and $x,y\in X\setminus B_d(p,r)$, 
  then there exists a continuum $E\sub X$ with $x,y\in E$ and
  $E\sub X\setminus B_d(p,r/C).$
\end{enumerate}

It is easy to see that \mbox{LLC1} is satisfied if and
only if $X$ is of bounded turning (as defined in \eqref{eq:bddturn}).

The  space $(X,d)$ is 
called   {\em annularly linearly locally connected}\index{annularly linearly locally connected (ALLC)} 
(abbreviated as ALLC)
if there exists a constant $C\ge 1$ with the following property:
  if  $p\in X$, $r>0$, and $x,y\in \overline B_d(p,2r)\setminus B_d(p,r)$, then 
  there exists a path $\gamma$ in $X$ joining $x$ and $y$ with 
  \begin{equation*}
    \gamma\sub 
    \overline B_d(p,Cr)\setminus B_d(p,r/C).
    \end{equation*}

The following proposition shows that the visual 
sphere  of  an expanding Thurston map is linearly locally connected
and annularly linearly locally connected.

\begin{prop}
  \label{prop:annLLC}
  \index{visual sphere}
  \index{visual metric}
  \index{metric!visual}
Let $f\:S^2\ra S^2$ be an expanding Thurston map, and $\varrho$ be a visual 
metric  for $f$. Then the following statements are true:

\begin{enumerate}
\item
  \label{item:bt} 
  $(S^2, \varrho)$ is of bounded turning.
  \index{bounded turning}

\item
  \label{item:llc_ann}
  $(S^2,\varrho)$ is annularly linearly locally connected. 
 
\item 
  \label{item:visSphLLC} 
  $(S^2,\varrho)$ is linearly locally connected. 

\end{enumerate}
\end{prop}

The statements \ref{item:bt}--\ref{item:visSphLLC} are not
logically independent, but one can show the implications
\ref{item:llc_ann} $\Rightarrow$ \ref{item:visSphLLC} $\Rightarrow$
\ref{item:bt} for quite general spaces. The ensuing proof will
not rely on this directly.
  
\begin{proof}  Let $\Lambda>1$ be the expansion factor of $\varrho$.
Then for some Jordan curve $\CC\sub S^2$ with $\post(f)\sub \CC$ we have $\varrho(x,y)\asymp \Lambda^{-m(x,y)}$ for all $x,y\in S^2$, where $m(x,y)=m_{f,\CC}(x,y)$ (see Definitions~\ref{def:mxy} and \ref{def:visual}). In the following,  all  cells  will be  for $(f,\CC)$ and all metric concepts refer to $\varrho$. 

\smallskip 
\ref{item:bt}
Let $x,y\in S^2$ be arbitrary. If $x=y$ there nothing to
prove. So we may assume that $x\ne y$. Let $n=m(x,y)\in \N_0$.
Then there exist $n$-tiles $X$ and $Y$ with $x\in X$, $y\in Y$, and $X\cap Y\ne \emptyset$. Since $X$ and $Y$ are Jordan regions, we can find a path $\alpha$ in $X\cup Y$ that joins $x$ and $y$.
Then by Proposition~\ref{lem:expoexp}~\ref{item:expoex2} we have 
$$  \diam(\alpha)\le \diam (X)+\diam(Y)
  \lesssim 
  \Lambda^{-n}\asymp \varrho(x,y). $$
  In particular,
  \begin{equation}
  \label{eq:S2bd_turn}
     \diam(\alpha)          \leq K \varrho(x,y)
\end{equation}
with a constant $K\geq 1$  independent of
$x$ and $y$.
Statement \ref{item:bt} follows.

\smallskip 
\ref{item:llc_ann}
Let  $p\in S^2$, $r>0$, and $x,y\in \overline B(p,2r)\setminus
B(p,r)$.  In the following, all implicit multiplicative constants will
be independent  of these  initial choices of $p$, $r$, $x$, and $y$.  
We have to find a path $\gamma$  joining $x$ and $y$ such that 
\begin{equation}\label{eq:anncont}
\ga \sub \overline B(p,Cr)\setminus B(p, r/C),
\end{equation} 
where $C\ge 1$ is a suitable constant.
The ensuing construction of  $\gamma$ is illustrated in
Figure~\ref{fig:allc}. 

Define 
$$n\coloneqq \max\{m(p,x), m(p,y)\}+1.$$
Then 
$$\Lambda^{-n}\asymp \min\{\varrho(p,x), \varrho(p,y)\}\asymp r. $$
Let $X,Y,Z$ be $n$-tiles with $x\in X$, $y\in Y$, and $p\in Z$.
Then by definition of $n$ we have 
$X\cap Z=\emptyset$ and $Y\cap Z=\emptyset$.

Since $f$ is expanding, we can choose $k_0\in \N_0$ as in \eqref{def:k0}.
In particular,  every connected set of $k_0$-tiles joining opposite sides of $\CC$ must contain at least ten $k_0$-tiles. 

Consider the set $U^{n+k_0}(p)$ as defined in
\eqref{eq:defUk}. This is the set of all $(n+k_0)$-tiles that
intersect an $(n+k_0)$-tile containing $p$.  Then
$f^n(U^{n+k_0}(p))$ is connected, and consists of
$k_0$-tiles. This set cannot join opposite sides of $\CC$; for
otherwise, we could find a connected set 
consisting of at most six
$k_0$-tiles with
this property (see the proof Lemma~\ref{lem:quasiball} for a
similar reasoning). This is impossible by definition of $k_0$.
Hence $f^n(U^{n+k_0}(p))$ is contained in a $0$-flower
(Lemma~\ref{lem:floweropp}) which implies that $U^{n+k_0}(p)$ is
contained in an $n$-flower
(Lemma~\ref{lem:mapflowers}~\ref{item:mapflowers3}). So there
exists an $n$-vertex $v$ with $p\in U^{n+k_0}(p)\sub W^n(v)$.
Since $Z$ contains $p$, this tile must be one of the $n$-tiles
forming the cycle of $v$. So $v\in Z$, and $v\notin X,Y$. This in
turn implies that $X$ and $Y$ do not meet $W^n(v)$ (see
Lemma~\ref{lem:flowerprop}~\ref{item:flower_prop3}).

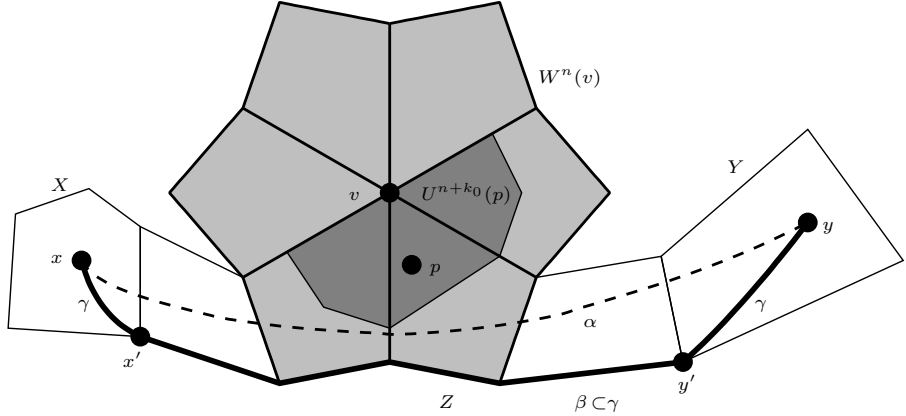
\begin{figure}
  \centering
  \begin{overpic}
    [width=12cm, tics=10,
    ]{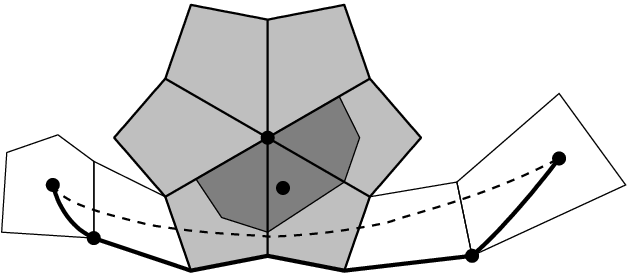}
    \put(5,14){${\scriptstyle x}$}
    \put(13,2.5){${\scriptstyle x'}$}
    \put(5,22){${\scriptstyle X}$}
    \put(90.5,17.7){${\scriptstyle y}$}
    \put(74.5,0){${\scriptstyle y'}$}
    \put(80,24){${\scriptstyle Y}$}
    \put(63,-2){${\scriptstyle \beta\,\subset \gamma}$}
    \put(83,9){${\scriptstyle \gamma}$}
    \put(8,9){${\scriptstyle \gamma}$}
    \put(64,7){${\scriptstyle \alpha}$}
    \put(38,21){${\scriptstyle v}$}
    \put(47,13){${\scriptstyle p}$}
    \put(48,-2){${\scriptstyle Z}$}
    \put(46,21){${\scriptstyle U^{n+k_0}(p)}$}
    \put(59,34){${\scriptstyle W^n(v)}$}    
  \end{overpic}
  \caption{Construction of $\gamma$.}
  \label{fig:allc}
\end{figure}

Pick a path $\alpha$ in $S^2$ that joins $x$ and $y$ and
satisfies \eqref{eq:S2bd_turn}. By Lemma~\ref{lem:echain} we can find a set $M$ of
$n$-tiles that forms an $e$-chain joining $X$ and $Y$ (see
Definition~\ref{def:e-chain}) so that
each tile in $M$ has non-empty intersection with $\alpha$. 
Pick $n$-vertices $x'\in  \partial X$, $y'\in \partial Y$. Since $X$ and $Y$ do not contain $v$, we have $x',y'\ne v$. 
Consider the graph $G_M= \bigcup_{U\in M}\partial U$ defined as
in \eqref{eq:def_GM} for the cell decomposition $\DD=\DD^n(f, \CC)$.  It consists of $n$-edges, is connected, has no cut points (Lemma~\ref{lem:nocut}), and contains $x'$ and $y'$ as vertices.  Hence 
there exists an edge path in $G_M$ joining $x'$ and $y'$ whose underlying 
set $\beta$ does not contain $v$. Then this edge path does not contain any 
edge in the cycle of $v$ and so $\beta\cap W^n(v)=\emptyset$. 
Let $\gamma$ be the path in $S^2$ that is obtained by running from $x$ to $x'$ along some path in $X$, then from $x'$ to $y'$ along $\beta$, and then from $y'$ to $y$ along some path in $Y$. Then $\gamma$ joins $x$ and $y$.

Since the sets $X,Y,\beta$ have empty intersection with $W^n(v)$ and hence with $U^{n+k_0}(p)$, it follows that $\gamma\cap U^{n+k_0}(p)=\emptyset$.
Thus  by Lemma~\ref{lem:UmB}~\ref{item:UmB1} we have
$$ \dist(p,\gamma)\gtrsim \Lambda^{-(n+k_0)}\asymp\Lambda^{-n}\asymp
r. $$
Hence there exists a constant $C_1\ge 1$ independent of the initial choices such that 
$$\gamma\cap B(p, r/C_1)=\emptyset. $$ 

The set $\gamma$ can be covered by $n$-tiles that meet $\alpha$.
Since 
$$ \diam (\alpha)\le K\varrho(x,y)\le 4Kr\lesssim r, $$ and 
$$\max\{\diam(U): U\text{ is an $n$-tile}\}\lesssim \Lambda^{-n}\asymp r, $$ we conclude that 
$$\diam (\gamma)\le \diam(\alpha)+2\max\{\diam(U): U\text{ is an $n$-tile}\}\lesssim r.$$
Since the initial point $x$ of $\gamma$ has distance $\le 2r$ from $p$, it follows that there exists a constant $C_2\ge 1$
independent of the initial choices such that 
$\gamma\sub B(p, C_2r). $ 
 If we set $C=\max\{C_1, C_2\}$, then  \eqref{eq:anncont} follows.
 
 \smallskip 
 \ref{item:visSphLLC} 
 To show that $(S^2, \varrho)$ is
 linearly locally connected, we verify the two relevant conditions
 \mbox{LLC1} and  \mbox{LLC2}; 
 here we can use possibly different constants $C$ in each of the
 conditions.
 
 Let $p\in S^2$, $r>0$, and $x,y\in B(p,r)$ 
be arbitrary. 
We  choose a path $\alpha$ joining $x$ and $y$ that satisfies \eqref{eq:S2bd_turn}. Define $E\coloneqq \alpha$ and $C=2K+1$. Then $x,y\in E$, and, since $\diam(\alpha)\le K\varrho(x,y)\le 2Kr$,
 we have 
 $$E\sub B(p, r+\diam(\alpha))\sub B(p, Cr). $$
 This shows that $(S^2, \varrho)$ satisfies  \mbox{LLC1}.
 
 In order to prove \mbox{LLC2}, let $p\in S^2$, $r>0$, and
 $x,y\in S^2\setminus B(p,r)$ 
 be arbitrary.  Let $\alpha$ be a path in $S^2$ joining $x$ and
 $y$. If $\alpha\cap B(p,r)=\emptyset$, define
 $E\coloneqq \alpha$.  Then $E$ is a continuum with $x,y\in E$
 and $E\sub S^2\setminus B(p,r)$.

 If $\alpha$ meets $B(p,r)$, then, as we travel from $x$ to $y$ along $\alpha$, there exists a first point with $x'\in \overline B(p,2r)$. Note that if $x\in \overline B(p,2r)$, then 
 $x'=x$, and $d(p,x')=2r$ otherwise. In any case,  $x'\in \overline B(p,2r)\setminus 
 B(p,r)$. Let $\alpha_x$ be the subpath of $\alpha$ obtained by traveling 
 along $\alpha$ starting from $x$ until we reach $x'$. 
 Then $\alpha_x\sub S^2\setminus B(p, r)$.  

Traveling along $\alpha$ in the opposite direction starting from $y$, we  similarly define a point $y'\in \overline B(p,2r)\setminus 
 B(p,r)$ and a subpath $\alpha_y\sub  S^2\setminus B(p, r)$ of $\alpha$ joining $y$ and $y'$. 
 Then $x',y'\in \overline B(p, 2r)\setminus B(p,r)$. Hence by \ref{item:llc_ann}
 there exists a path $\gamma$ in $S^2$ that joins $x'$ and $y'$ and satisfies $\gamma\sub S^2\setminus B(p,r/C)$.  Here $C\ge 1$ is a constant 
 independent of the initial choices.
Now define $E=\alpha_x\cup \gamma\cup \alpha_y$. Then $E$ is a continuum  with $x,y\in E$ and $E\sub S^2\setminus B(p,r/C)$.
It follows that \mbox{LLC2} is satisfied as well. 
 \end{proof}

\section{Doubling and Ahlfors regularity} \label{sec:periodic} 
 Here part \ref{item:S2doubling} of
Theorem~\ref{thm:S2vsf} and  Proposition~\ref{prop:Ahlforsreg}
are proved. We first need some preparation.

Let $f\: S^2\ra S^2$ be a branched covering map  on a $2$-sphere $S^2$.  Recall that 
a point $p\in S^2$ is called periodic if there exists $n\in \N$ such that $f^n(p)=p$ and that the smallest $n$ for which this is true
is called the period of the periodic point.

The following lemma is essentially well known.

\begin{lemma} \label{lem:cycle} 
Let $f\:S^2\ra S^2$ be a branched covering map. 
Then $f$ has no periodic critical points if and only if there exists $N\in \N$ such that 
$$ \deg(f^n,p)\le N$$ 
for all $p\in S^2$ and all $n\in \N$. 
\end{lemma}

\begin{proof}
Note that for $p\in S^2$ and $n\in \N$
we have 
\begin{equation}\label{eq:degfn}
  \deg(f^n,p)=\prod_{k=0}^{n-1}\deg(f, f^{k}(p)).
\end{equation}
So if $p$ is a periodic critical point of period $l$, say, and $d=\deg(f,p)\ge 2$, then 
$$\deg(f^n,p) \ge d^{ \lceil n/l \rceil} \ge 2^{n/l} \to \infty$$ as 
$n\to \infty$. Hence $\deg(f^n,p)$ is not uniformly bounded.

If $f$ has no periodic critical points, then the orbit $p$, $f(p)$, $f^2(p)$, $\dots$ of a  point $p\in S^2$ can contain each critical point at most once.
Hence by \eqref {eq:degfn}
we have 
$$ \deg(f^n,p)\le N\coloneqq \prod_{c\in \crit(f)} \deg(f,c). $$
Note that the last product is finite, because $f$ has only finitely
many critical points. 
\end{proof} 

 

\begin{cor}
  \label{cor:no_crit_flowers}
  Let $f\colon S^2\to S^2$ be a Thurston map with no periodic
  critical points. Then there is a constant $N\in \N$ with the following property: 
  if $\CC\sub S^2$ is a Jordan curve with $\post(f)\sub \CC$, then for each $n\in \N_0$ and each vertex $v$ in $\DD^n(f,\CC)$ the cycle of $v$ 
  has length at most $N$. 
\end{cor}

In other words, the closure of each $n$-flower $W^n(v)$ contains at most $N$ tiles of level $n$, where $N$ only depends on $f$. 

\begin{proof}
  By 
  Lemma~\ref{lem:cycle} there exists $N'\in \N$ such that
the inequa\-lity   $\deg(f^n,p)\le N'$ is valid for all $p\in S^2$ and $n\in \N_0$.  This
  implies that the cycle of each $n$-vertex (defined with respect to any 
  Jordan curve $\CC\sub S^2$ with $\post(f)\sub \CC$) has length at most 
    $N\coloneqq 2N'$ 
  (see Lemma~\ref{lem:flowerprop}~\ref{item:flower_prop1}). 
\end{proof}

We are now ready to prove the first part of  
Theorem~\ref{thm:S2vsf}. 

\begin{proof}[Proof of
  Theorem~\ref{thm:S2vsf}~\ref{item:S2doubling}]
 Assume first that $f$ has no periodic critical points.  By Theorem~\ref{thm:main}
  there exists an iterate $F=f^n$ and an  $F$-invariant Jordan curve $\CC\sub S^2$ with $\post(f)=\post(F)\sub \CC$.  Then $F$ is also an expanding Thurston map (Lemma~\ref{lem:Thiterates}) and it has no periodic critical points as easily follows from \eqref{eq:critpfn}. So by Corollary~\ref{cor:no_crit_flowers} there exists a number $N\in \N$ such that the cycle of  each vertex in $\DD^n(F,\CC)$, $n\in \N_0$,  has length $\le N$.

  It suffices to show that $S^2$ equipped  with a visual metric
  for $F$ is doubling, since the class of visual metrics for $f$ and $F$ agree (Proposition~\ref{prop:visualsummary}~\ref{item:vsfF}). 
  
  Fix such a visual metric $\varrho$  for $F$, and denote by $\Lambda>1$ its expansion factor. In the following, all  cells 
are   for $(F,\CC)$ and all metric notions refer to $\varrho$. 
  To establish that $S^2$ is doubling, we now proceed as in
the proof of Theorem~\ref{thm:Cquasicircle}.

Let $x\in S^2$ and $0<r\le 2 \diam (S^2)$ be arbitrary. 
We have to cover $B(x,r)$ by a controlled number of sets of diameter $<r/4$. 
Using Proposition~\ref{lem:expoexp}, we can find $n\in
\N_0$ depending on $r$, 
as well as  constants $C(\asymp)>0$ and  $k_0\in \N_0$ independent of $x$ 
and $r$  with the following
properties: 


\begin{enumerate}

\item
  \label{item:visual_cover1} 
  $r\asymp \Lambda^{-n}$.

\item
  \label{item:visual_cover2} 
  $\diam(X)< r/4$, whenever $X$ is an $(n+k_0)$-tile.

\item
  \label{item:visual_cover3} 
  $\dist(X,Y)\ge r$, whenever $n-k_0\ge 0$ and $X,Y$ are disjoint $(n-k_0)$-tiles. 

\end{enumerate} 

Let $T$ be the set of all $(n+k_0)$-tiles that meet $B(x,r)$. Then the collection $T$ forms a cover of  $B(x,r)$ and consists of sets of diameter $<r/4$ by \ref{item:visual_cover2}. 
Hence it suffices to find a uniform upper bound for $\#T$, independent of $x$ and $r$. 
If $n< k_0$, then $\#T\le 2\deg(F)^{2k_0} $ (see Proposition~\ref{prop:celldecomp}~\ref{item:noVEX}) and we have such a bound.

Otherwise, $n-k_0\ge 0$. Pick an  $(n-k_0)$-tile   $X$ with $x\in X$. 
If $Z$ is  an arbitrary $(n+k_0)$-tile  in $T$, then we    can find  a unique  $(n-k_0)$-tile $Y$ that contains $Z$ (here we use that $\CC$ is $F$-invariant and so each 
tile is subdivided by tiles of higher levels).  

There exists a point $y\in Z\cap B(x,r)$.
Hence $\dist(X,Y)\le \varrho(x,y)<r$. This implies $X\cap Y\ne \emptyset$ by \ref{item:visual_cover3}.
So whatever $Z\in T$ is, the corresponding $(n-k_0)$-tile $Y\supset Z$ meets the fixed $(n-k_0)$-tile $X$. Hence $Y$ must share an 
$(n-k_0)$-vertex $v$ with $X$ which implies  $Y\sub \overline{W^{n-k_0}(v)}$. Since by choice of $N$ the set $\overline{W^{n-k_0}(v)}$ contains at most $N$  tiles of level  $(n-k_0)$, and the number of  $(n-k_0)$-vertices in $X$ is equal to   $\#\post(F)$,   this leaves at most $N\#\post(F)$ possibilities for $Y$. 

Since every $(n-k_0)$-tile contains at most 
$2\deg(F)^{2k_0}$ tiles of level $(n+k_0)$, it follows that 
$\#T\le 2N\#\post(F)\deg(F)^{2k_0}$. 
So we get a uniform bound as desired, which shows that $(S^2,\varrho)$ is
doubling. 

  To show the reverse implication, we use
  the following fact about doubling spaces, which is easy to show: in every ball there cannot be too many pairwise disjoint smaller balls that all have the same radius. More precisely,  for every $\eta\in (0,1)$ there is a number $K$ such that  every open ball of radius $r$ contains at  most $K$    pairwise disjoint open balls of radius $\eta r$. 
  
  Now suppose $f\: S^2\ra S^2$ is an expanding Thurston map such that  $S^2$ equipped with some visual metric for $f$ is doubling. Pick a Jordan curve $\CC\sub S^2$ with $\post(f)\sub \CC$. In the following, cells will be for $(f,\CC)$.  Let $p\in S^2$ and $n\in\N$.  In order to show that $f$ has no periodic critical points, it suffices to give a uniform bound on $d=\deg(f^n,p)$ (see Lemma~\ref{lem:cycle}).
  For this we may assume that $\deg(f^n,p)\ge 2$. Then $p$ is an $n$-vertex and the closure of the $n$-flower $W^n(p)$ consists of precisely $2\deg(f^n,p)$ tiles
  of level $n$. These $n$-tiles have 
pairwise disjoint interiors 
and each interior  contains a ball of radius $r\asymp \Lambda^{-n}$ (see Lemma~\ref{lem:quasiball}). On the other hand, $\diam(\overline{W^n(p)})\lesssim \Lambda^{-n}$. Since $S^2$ is doubling,  it follows that the number of these tiles and hence $\deg(f^n,p)$ is uniformly bounded from above by a constant independent of $p$ and $n$.
  Hence $f$ has no periodic critical points. 
  \end{proof}

We now prove the Ahlfors regularity of $(S^2,\varrho)$  when $f$
has no periodic critical points.

\begin{proof}[Proof of Proposition~\ref{prop:Ahlforsreg}]  
  As in the statement, let $f\colon S^2\to S^2$ be  an expanding
  Thurston map without periodic critical points, $\varrho$ be a
  visual metric for $f$, and $\nu=\nu_f$ be its  measure of maximal
  entropy. Assume that $\Lambda>1$ is the expansion
  factor of $\varrho$.

  By Theorem~\ref{thm:main} we can fix an iterate $F=f^n$ and an
  $F$-invariant Jordan curve $\CC\sub S^2$ with
  $\post(f)\subset \CC$. The map $F$ is an expanding Thurston map by
  Lemma~\ref{lem:Thiterates}, and $\varrho$ is a visual metric for
  $F$ with expansion factor $\Lambda_F \coloneqq \Lambda^n$ by
  Proposition~\ref{prop:visualsummary}~\ref{item:vsfF}.  In the
  following, cells are defined for $(F,\CC)$ and metric notions
  refer to $\varrho$.  The measure $\nu$ is also the measure of
  maximal entropy $\nu_F$ for $F$ (see
  Proposition~\ref{prop:exmeasure} and Theorem~\ref{thm:nuF}).

Let  $\overline B(x,R)\sub S^2 $ be an  arbitrary closed ball, where
  $x\in S^2$ and $0<R\le \diam(S^2)$. We use the sets $U^m(x)$ as
  defined in \eqref{eq:defUk} for the map $F$.
  Since $F$ does not have periodic critical points, the length of
  the cycle of each vertex is uniformly bounded
  (Corollary~\ref{cor:no_crit_flowers}). This implies that for each $m\ge 0$ the 
  set $U^m(x)$ consists of a uniformly bounded number of 
  $m$-tiles (by definition $U^m(x)=S^2$ for $m<0$).
  
   Since the sets $U^m(x)$ are closed, Lemma~\ref{lem:UmB}~\ref{item:UmB2} (applied to $F$)
gives the inclusions 
  \begin{equation*}
    U^{m+n_0}(x)\subset \overline B(x,R) \subset U^{m-n_0}(x),
  \end{equation*}
  where $m=\left\lceil -\log (R)/\log(\Lambda_F )\right\rceil$ and
  $n_0\in \N_0$ is a constant independent of the ball. Noting that
  \begin{equation*}
    Q
    \coloneqq
    \frac{\log(\deg(f))}{\log(\Lambda)} 
    =
    \frac{\log(\deg(F))}{\log(\Lambda_F )}
  \end{equation*}
  and using Proposition~\ref{prop:exmeasure}, we conclude 
  \begin{align*} 
    \nu_F(U^{m+n_0}(x))
    &\asymp 
              \nu_F(U^{m-n_0}(x)) \asymp  \deg(F)^{-m}\\
    &\asymp 
             \exp\big(\log(R)\log(\deg(F))/\log(\Lambda_F )\big)=R^Q,  
  \end{align*}
  and so  $\nu(\overline{B}(x,R))=\nu_F(\overline B(x,R)) \asymp
  R^Q$. Here the constants $C(\asymp)$ are independent of the
  ball. The Ahlfors  $Q$-regularity of $(S^2,\varrho, \nu)$
  follows.   

  This in turn implies that
  $\nu(M)\asymp\mathcal{H}^Q_\varrho(M)$ for every Borel set
  $M\sub S^2$, where $C(\asymp)$ is independent of $M$. Since
  $\nu$ is a probability measure, we conclude that
  $0<\mathcal{H}^Q_\varrho(S^2)<\infty$. It also follows that the
  Hausdorff dimension of $(S^2,\varrho)$ is equal to $Q$.
\end{proof}

\section{Quasisymmetry and rational Thurston maps}
\label{sec:rational-maps}

In this section we prove Theorem~\ref{thm:S2vsf}~\ref{item:S2qsphere}. We will also 
derive Theorem~\ref{thm:main01} as a consequence of this and other previous results.   The proof of  Theorem~\ref{thm:S2vsf}~\ref{item:S2qsphere}
mostly follows  \cite{Me02} and \cite{Me08}. It was
independently established in \cite{HP} by a different method. 

The more difficult implication  in   Theorem~\ref{thm:S2vsf}~\ref{item:S2qsphere}
amounts to proving  that if a rational Thurston map $f\: \CDach \ra \CDach$ is expanding, then its visual sphere is a quasisphere. For this one wants to show that 
the chordal metric $\sigma$ on $\CDach$ is quasisymmetrically equivalent to each visual metric $\varrho$ for $f$. The key for this is to analyze the metric properties of the cell decompositions $\DD^n(f,\CC)$ with respect to the chordal metric
$\sigma$. 
This is of independent interest and we record the relevant facts
in a separate statement. A similar approach to prove
quasisymmetric equivalence can be found in
\cite[Theorem~3.4]{Ki14}. 

\begin{prop}[Tiles in the chordal metric]
  \label{prop:chordalm}
  \index{metric!chordal}
  \index{chordal metric}
  \index{tile}
  Let $f\colon \CDach\to \CDach$ be a rational  Thurston
  map without periodic critical points, and 
   $\CC\sub \CDach$ be a Jordan curve with 
  $\post(f)\sub \CC$. We equip $\CDach$ with the chordal metric $\sigma$ and 
   denote by ${\bf X}^n$ for $n\in \N_0$ the collection of 
  $n$-tiles for $(f,\CC)$. Then for all $n,k\in \N_0$ the following statements are true: 
   
  \begin{enumerate}
   \item 
    \label{item:diamXY}
    If  $X,Y\in \X^n$  and  $X\cap Y\neq \emptyset$,  then
    \begin{equation*}
      \diam (X) \asymp \diam (Y). 
    \end{equation*}
  \item 
    \label{item:distdiamX}
   If  $X,Y\in \X^n$  and  $X\cap Y= \emptyset$,  then
    \begin{equation*}
      \dist(X,Y)\gtrsim \diam (X).
    \end{equation*}
   \item 
    \label{item:diamXkXkm1}
    If $X\in {\bf X}^n$,  $Y\in {\bf X}^{n+k}$, and $X\cap Y\ne \emptyset$,  then
    \begin{equation*}
      \diam (X) \gtrsim \diam (Y). 
    \end{equation*}
     \item 
    \label{item:diamXkXkm2}
    If  $X\in {\bf X}^n$,  $Y\in {\bf X}^{n+k}$, and $X\cap Y\ne \emptyset$,  then
    \begin{equation*}
      \diam (X) \lesssim \diam (Y),
    \end{equation*}
    where $C(\lesssim)= C(k)$.
    \item
     \label{item:diamXxprime}
    Let $\widetilde{\CC}\subset\CDach$ be  another Jordan curve with
    $\post(f)\subset \widetilde{\CC}$. If $n\in \N_0$, $X$ is an $n$-tile for
    $(f,\CC)$,  $\widetilde{X}$ is  an $n$-tile for
    $(f,\widetilde{\CC})$, and $X\cap\widetilde{X}\neq
    \emptyset$, then
    \begin{equation*}
      \diam (X) \asymp \diam (\widetilde{X}).
    \end{equation*}  
     \item 
    \label{item:sxydiamX}
  If  $x,y\in \CDach$ and $x\ne y$, then
    \begin{equation*}
      \sigma(x,y)\asymp \diam (X),    \end{equation*}
   whenever $X\in {\bf X}^m$ contains $x$, where $m=m_{f,\CC}(x,y)$. 

  \end{enumerate}
  The implicit multiplicative constant  in
  \ref{item:diamXY}, \ref{item:distdiamX}, \ref{item:diamXkXkm1},  \ref{item:diamXxprime}
 is  independent of the tiles involved in the inequalities and their levels,  in \ref{item:sxydiamX} is independent of $x$, $y$, $X$, and in  \ref{item:diamXkXkm2} only depends on $k$.  
     \end{prop}

So in the previous inequalities  we can choose all  implicit
multiplicative 
constants only depending on $f$ and $\CC$ 
(and on $\widetilde \CC$ in \ref{item:diamXxprime})
with the exception of \ref{item:diamXkXkm2}, where we also have dependence on the level difference $k$ of the tiles (but not on the specific tiles). 
Statements \ref{item:diamXkXkm1} and \ref{item:diamXkXkm2} essentially 
 say that if  two tiles for $(f,\CC)$ have  non-empty intersection, then their 
  diameters are comparable in the following way:  the diameter
 of the lower-level tile   bounds the diameter of the higher-level tile up to a uniform constant, while in the converse inequality the constant depends only on the level difference.

The proof of the previous proposition will   be based on  
 two observations. First, the unions of two sets of tiles with the 
same combinatorics are conformally equivalent. Second,
there are only {finitely many}
different {combinatorial types}, because the local degrees of all iterates of the map are uniformly bounded. 
The statements can then be derived from Koebe 
distortion estimates for conformal maps. 

The distortion estimates we will need are formulated in   Lemma~\ref{lem:distest}. There they are stated in terms of the chordal metric and 
spherical derivatives of the conformal map. A subtlety here is that in order to get uniform estimates one has to assume that the image of the map is not too large. In 
Lemma~\ref{lem:distest} we stipulate  that the image of the map is contained in a hemisphere, i.e., a chordal disk of radius $\sqrt 2$.
So when we apply Lemma~\ref{lem:distest} we have to make sure that this hypothesis 
is true (see also the discussion after the proof of Theorem~\ref{thm:Koebe}).

We now  fill in the details of this outline. 
Let $\DD$ be a cell complex. A subset $\DD'\sub \DD$ is called a {\em
  subcomplex}\index{subcomplex} of $\DD$ if the following condition is true: if $\tau\in
\DD'$, $\sigma \in \DD$, and  $\sigma\sub \tau$, then $\sigma\in
\DD'$. If $\DD'$ is a subcomplex of $\DD$, then the cells in $\DD'$
form a cell decomposition of the underlying set 
$$ |\DD'|\coloneqq \bigcup\{c:c\in \DD'\}. $$

Let $\DD$, $\DD'$, $\widetilde \DD$  be cell complexes and suppose we have labelings  $L\: \DD \ra \widetilde \DD$ 
and $L'\: \DD'  \ra \widetilde \DD$ (see Definition~\ref{def:labeldecomp}). Then we say that an isomorphism 
$\phi\: \DD\ra \DD'$ of cell complexes (see Definition~\ref{def:compiso}) 
is 
{\em label-preserving}\index{label-preserving isomorphism}\index{isomorphism!of cell complexes!label-preserving} 
if $L(\tau)=L'(\phi(\tau))$ for each $\tau \in \DD$.

 Now suppose that $f\: \CDach \ra \CDach$ is a rational  Thurston map, and $\CC\sub \CDach$ is a Jordan curve with $\post(f)\sub \CC$.
We consider the cell decompositions $\DD^n=\DD^n(f,\CC)$ of 
$\CDach$ for $n\in \N_0$. If $\tau\in \DD^n$, then $f^n|\tau$ is a homeomorphism of 
the $n$-cell $\tau$ onto the $0$-cell $f^n(\tau)$. So the map
$\tau\mapsto f^n(\tau)$ induces a labeling $\DD^n\ra \DD^0$.
We call this the 
{\em natural labeling}\index{natural labeling}\index{cell!complex!natural labeling of}\index{labeling!natural} on $\DD^n$. Similarly, 
the map  $\tau\mapsto f^n(\tau)$ induces a natural labeling on every
subcomplex of $\DD^n$.  
  
\begin{lemma}\label{lem:confequiv}Let $n,m\in \N_0$,  $\DD$ be a subcomplex of $\DD^n$, and $\DD'$ be a subcomplex of $\DD^m$ both equipped with their natural labelings. If $\phi\: \DD\ra \DD'$ is a label-preserving isomorphism, then there exists 
a homeomorphism $h\: |\DD|\ra |\DD'|$ such that

\begin{enumerate}

\item
  \label{item:confeq1}
  $h(\tau)=\phi(\tau)$ for each $\tau \in \DD$, 

\item
  \label{item:confeq2}
  $h$  maps $\inte(|\DD|)$ conformally onto    $\inte(|\DD'|)$. 
\end{enumerate}
\end{lemma}

Here $\inte(|\DD|)$ and    $\inte(|\DD'|)$ denote the interiors of 
$|\DD|$ and    $|\DD'|$, respectively, as subsets of $\CDach$. 
Roughly speaking, the lemma says  that combinatorial equivalence of two subcomplexes $\DD$ and $\DD'$ 
gives conformal equivalence of their underlying sets.


\begin{proof} If $\tau\in \DD$, then $f^n(\tau)=f^m(\phi(\tau))$, because 
$\phi$ is label-preserving. Hence we can define a homeomorphism
$h_\tau\coloneqq (f^m|\phi(\tau))^{-1}\circ (f^n|\tau)$
 of $\tau $ onto $\phi(\tau)$.
It is clear that the maps $h_\tau$ are compatible under inclusions of cells: if $\sigma,\tau\in \DD$ and $\sigma\sub \tau$, then $h_\tau|\sigma=h_\sigma$.
We now define a map $h\: |\DD|\ra |\DD'|$ as follows. For  $p\in |\DD|$ pick $\tau\in \DD$ with $p\in \tau$. Set $h(p)\coloneqq h_\tau(p)$.  As in the proof of Proposition~\ref{prop:thurstonex} one sees that $h$ is well-defined and as in the proof of Lemma~\ref{lem:isocellhomeo}~\ref{item:isocellhomeo1} that $h$ is a homeomorphism of $|\DD|$ onto $|\DD'|$. Obviously, $h$ has property \ref{item:confeq1}.

To establish property \ref{item:confeq2}, first note that $h$ maps $\inte(|\DD|)$ homeomorphically onto  $\inte(|\DD'|)$ (this follows from  the 
``invariance of domain''; see for example \cite[Theorem 2B.3, p.~172]{Ha}).  So it 
 suffices to show that $h$
is holomorphic   on $U\coloneqq \inte(|\DD|) \sub \CDach$. 

The definition of $h$ implies that $f^n=f^m\circ h$ on $U$. 
So if $p\in U$ and $q=h(p)$ is not a critical point of $f^m$, then there exists on open neighborhood $V$ of $q$ where $f^m|V$ has a holomorphic inverse $(f^m|V)^{-1}$. Then $h=(f^m|V)^{-1} \circ f^n$ near $p$, which shows that $h$ is holomorphic 
near $p$. This implies that $h$ is holomorphic on $U\setminus h^{-1}(\crit(f^m))$.
The finitely many points in $h^{-1}(\crit(f^m))\cap U$ are removable singularities for $h$, because $h$ is continuous. It follows that $h$ is holomorphic on $U$ as desired. 
%
%
 \end{proof}

\begin{proof}
  [Proof of Proposition~\ref{prop:chordalm}] Unless otherwise stated, in the following all cells are for $(f,\CC)$. 
  
  \smallskip
  \ref{item:diamXY}   For $n\in \N_0$ and $X,Y\in \X^n$ with $X\cap Y \neq 
  \emptyset$, we  consider  the complex $\DD(X,Y)$, equipped with the
  natural labeling,  consisting of all $n$-cells $c$  for which there
  exists an $n$-tile $Z$ with $c\sub Z$ and $Z\cap (X\cup Y)\ne
  \emptyset$. 
  Obviously, 
  \begin{equation}\label{eq:defVkXY}
    |\DD(X,Y)|=\bigcup \{Z\in\X^n : Z\cap(X\cup Y)\neq \emptyset\}.  
  \end{equation}
  Let $\Om(X,Y)$ be the interior of $|\DD(X,Y)|$. Then $\Om(X,Y)$ is a
  region containing $X$ and $Y$.  
    
  Suppose that $X',Y'$  is a pair of  non-disjoint $m$-tiles, $m\in
  \N_0$. We call   $\DD(X,Y)$ and $\DD(X',Y')$ 
  equivalent if there 
  is
  a label-preserving isomorphism $\phi\:
  \DD(X,Y)\ra \DD(X',Y')$ with $\phi(X)=X'$ and $\phi(Y)=Y'$. 
  If  $\DD(X,Y)$ and $\DD(X',Y')$ are equivalent, then by
  Lemma~\ref{lem:confequiv} there exists a conformal map $h\:
  \Om(X,Y)\ra \Om(X',Y')$ with $h(X)=X'$ and $h(Y)=Y$. 

  Since $f$ has no periodic
  critical points, the length of the cycle of each vertex is uniformly bounded (Corollary~\ref{cor:no_crit_flowers}). So  each $n$-vertex  is contained in a uniformly bounded number of $n$-tiles independent of $n$. 
  This implies that the number of $n$-tiles, and hence the number of
  $n$-cells, in $\DD(X,Y)$ is uniformly bounded by a number
  independent of $X$, $Y$, and $n$. Therefore,  among the complexes
  $\DD(X,Y)$  there are only finitely many equivalence classes.
  Since $f$ is expanding, there are also  only finitely many
  complexes $\DD(X,Y)$ such that $\Om(X,Y)$ is not contained in a
  hemisphere. Hence we can find finitely many complexes $\DD(X_1,
  Y_1), \dots, \DD(X_N, Y_N)$ such that each complex $\DD(X,Y)$ not in
  this list is equivalent to one complex $\DD(X_i,Y_i)$ and such that
  $\Om(X,Y)$ is contained in a hemisphere. It follows from \eqref{di0}
  in Lemma~\ref{lem:distest} (applied to $A=X_i$, $B=Y_i$,
  $\Om=\Om(X_i,Y_i)$, and the conformal map $h\:\Om(X_i,Y_i)\ra
  \Om(X,Y)$ produced by Lemma~\ref{lem:confequiv}) that
  $\diam(X)\asymp \diam(Y)$ with $C(\asymp)$ independent of $X$, $Y$,
  and $n$.

  \smallskip 
  \ref{item:distdiamX}
  The argument is very similar to the previous one.
For $n\in \N_0$ and  $X\in \X^n$, 
  we consider the cell complex $\DD(X)$, equipped with the natural
  labeling, consisting of all $n$-cells $c$  for which there exists an
  $n$-tile $Z$ with $c\sub Z$ and $Z\cap X\ne \emptyset$. Then  
  \begin{equation} 
    |\DD(X)|=\bigcup \{Z\in\X^n : X\cap Z\neq \emptyset\}.   
  \end{equation}
  If we define $\Om(X)$ to be the interior 
  of $|\DD(X)|$, then $\Om(X)$ is a region that contains  $X$. Moreover, if $Y$ is an $n$-tile with $X\cap Y=\emptyset$, then  $\Om(X)$  is disjoint
  from $Y$. 
   
  If  $m\in \N_0$ and $X'$  is an $m$-tile, then we    call the
  complexes $\DD(X)$ and $\DD(X')$ equivalent if there exists a
  label-preserving isomorphism $\phi\: \DD(X)\ra \DD(X')$
  with $\phi(X)=X'$. 
  
  Again there are only finitely many equivalence classes of the complexes
  $\DD(X)$. Based on Lemma~\ref{lem:confequiv} and
   \eqref{eq:dhdiamhA}, \eqref{di1}
 in  Lemma~\ref{lem:distest}, we conclude that for each $n$-tile $Y$ with $X\cap Y=\emptyset$, we have 
  $$\dist(X,Y)\geq \dist(X, \partial\Om(X))\gtrsim \diam(X), $$ 
  where $C(\gtrsim)$ does not depend on $X$ and $Y$.

 \smallskip
  \ref{item:diamXkXkm1}--\ref{item:diamXkXkm2}  Let $k,n \in \N_0$, $X\in {\bf X}^n$, $Y\in {\bf X}^{n+k}$, and $X\cap Y\ne \emptyset$.   
  
 If $W$ is any $n$-flower, then any two $n$-tiles contained in $\overline 
 W$ have an $n$-vertex in common, and hence have comparable diameter by \ref{item:diamXY}. 
 This implies that $\diam(Z)\asymp \diam(W)$ whenever $Z$ is an $n$-tile with $Z\cap\overline W\ne \emptyset$. We also see that 
  $\diam(W)\asymp \diam(W')$, whenever  $W$ and $W'$ are 
   $n$-flowers with $W\cap W'\ne \emptyset$. 
 
Now by  Lemma~\ref{lem:tileflowerimprov} (applied for $\widetilde \CC=\CC$)
we can cover the $(n+k)$-tile 
  $Y$ with  $M$ $n$-flowers $W_1,\dots, W_M$, where $M$ is independent of 
     $Y$. Since $X$ and $Y$ have a point in common, we may 
     assume that $X\cap W_1\ne \emptyset$. Moreover,  
 since $Y$ is connected, we may assume that each  flower in the list meets one of the previous ones. Then
 $$ \diam(X)\asymp \diam(W_1)\asymp \diam(W_i) $$
 for $i=1, \dots, M$. This implies
 $$ \diam(Y)\le \sum_{i=1}^M\diam(W_i)\asymp\diam(W_1)\asymp\diam(X).$$
 Since in the previous inequalities all implicit multiplicative constants only depended on $f$ and $\CC$, claim \ref{item:diamXkXkm1} follows.

 For \ref{item:diamXkXkm2}  note that by choosing tiles that contain 
  a point in $X\cap Y$, we can find tiles $X^i$ of levels $i=n, \dots, n+k$ such that 
  $X=X^n$, $Y=X^{n+k}$, and  $X^{i}\cap X^{i+1}\ne \emptyset$ for $i=n, \dots, n+k-1$.
  If we can show that 
  $\diam(X^i)\lesssim \diam(X^{i+1})$ with a uniform constant $C(\lesssim)$ only depending on $f$ and $\CC$, then we conclude that $\diam(X)\lesssim  \diam(Y)$ with a constant
  $C(\lesssim)=C(k)$ as desired.
  
  In other words, we are reduced to proving the inequality $\diam(X)\lesssim \diam(Y)$ under the additional assumption that $Y\in  {\bf X}^{n+1}$, where
 the    implicit multiplicative constant is supposed to  depend only on $f$ and $\CC$. 

 By \ref{item:diamXY} and   Lemma~\ref{lem:difflevel}~\ref{item:coverflower2} the $n$-tile $X$  can be covered by a controlled number of $(n+1)$-flowers whose diameters  are comparable to $\diam(Y)$.  If one combines this with similar estimates as in the previous argument, then \ref{item:diamXkXkm2} immediately follows.

   \smallskip
  \ref{item:diamXxprime}
If  $X$ and $\widetilde{X}$
  are as  in the statement, then  by \ref{item:diamXY} and Lemma~\ref{lem:tileflower}   we can cover
  $\widetilde{X}$ by $M$ $n$-flowers $W_1, \dots,
  W_M$ for  $(f,\CC)$ with $\diam(W_i)\asymp \diam(X)$ for $i=1, \dots , M$.    Here $C(\asymp)$ and the number $M$ are
  independent of $n$, $X$, and $\widetilde{X}$.  By an estimate as in 
   the proof of    \ref{item:diamXkXkm1}, we conclude 
   $\diam(\widetilde X)\lesssim \diam(X)$.  For the other inequality we reverse the 
 roles of $X$ and $\widetilde{X}$ and use  an analog of 
  \ref{item:diamXY} for the curve $\widetilde \CC$.   
  
  \smallskip
  \ref{item:sxydiamX}
  Let $x,y\in \CDach$ with $x\ne y$  be arbitrary,  $m\coloneqq m_{f,\CC}(x,y)$, and $X$ be an $m$-tile with $x\in X$. By 
  Definition~\ref{def:mxy}  there are $m$-tiles $X^m$ and $Y^m$ with $x\in
  X^m$, $y\in Y^m$, and $X^m\cap Y^m\ne \emptyset$. Then 
   by \ref{item:diamXY} we have 
   $$\diam(X)\asymp \diam(X^m) \asymp \diam(Y^m),$$
   and so 
  \begin{equation*}
    \sigma(x,y)\leq \diam (X^m) + \diam(Y^m) \asymp \diam (X). 
  \end{equation*}
  
  On the other hand, pick  $(m+1)$-tiles
  $X^{m+1}$ and $Y^{m+1}$ with $x\in X^{m+1}$ and $y\in Y^{m+1}$. Then $X^{m+1}\cap Y^{m+1}=\emptyset$ by definition of $m$.  We have $x\in X\cap X^{m+1}$, and so 
  $\diam(X^{m+1})\asymp \diam(X)$ by \ref{item:diamXkXkm1} and 
  \ref{item:diamXkXkm2}. Then  \ref{item:distdiamX} 
  implies that 
  \begin{equation*}
    \sigma(x,y)\geq \dist(X^{m+1},Y^{m+1}) \gtrsim \diam (X^{m+1})\asymp \diam(X). 
  \end{equation*}
Since in  the previous inequalities all the implicit multiplicative constants can be chosen independently of $x$, $y$, and $X$, the statement follows. \end{proof}

%

We can now prove the following  crucial result.
 
\begin{lemma}\label{lem:sigqsvis}\label{cor:visualqsmetric}  Let $f\: \CDach\ra \CDach$ be a rational Thurston map with no periodic critical points. Then each visual metric $\varrho$ for $f$ is quasisymmetrically equivalent to the chordal metric $\sigma$.
\end{lemma} 

\begin{proof} 
The map $f$ is expanding by Proposition~\ref{prop:rationalexpch}. 
Let  $\varrho$ be a visual metric for $f$. We have to show  that  the identity map
 $    \id_{\CDach}\colon (\CDach,\varrho) \to (\CDach,\sigma)$ 
 is a quasisymmetry.   We will actually prove that this map  
  is weakly
    quasisymmetric:
     there 
 exists a constant $H\ge 1$ such that we have the implication
 \index{weak quasisymmetry}\index{quasisymmetry!weak}    
  \begin{equation}
    \label{eq:defweakqs}
    \varrho(x,z)\leq \varrho(x,y) \Rightarrow\sigma(x,z)\leq H \sigma(x,y)
  \end{equation}
  for all  $x,y,z\in \CDach$. A weak quasisymmetry between  connected
  doubling spaces is quasisymmetric (see
 Proposition~\ref{prop:wkqs}). We may apply this statement, because  $(\CDach,\sigma)$ is connected and doubling. Moreover, $(\CDach, \varrho)$ is connected, and it is also   doubling, because   $f$ has no periodic critical points 
 (see 
   Theorem~\ref{thm:S2vsf}~\ref{item:S2doubling}).
   
 We pick a Jordan curve $\CC\sub \CDach$ with $\post(f)\sub \CC$ and denote by $\Lambda>1$ the expansion factor of $\varrho$. In the following, we will consider tiles for $(f,\CC)$.
  
 Now suppose  that
  $x,y,z\in \CDach$ are points with 
 $\varrho(x,z)\le \varrho(x,y)$. 
 Define $n=m_{f,\CC}(x,y)$ and $l=m_{f,\CC}(x,z)$, where $m_{f,\CC}$ is as in Definition~\ref{def:mxy}.  Then 
 $$ \Lambda^{-l}\asymp \varrho(x,z)\le \varrho(x,y)\asymp  \Lambda^{-n}.$$
 Hence there exists a constant $k_0\in \N_0$  independent of 
 $x,y,z$ such that $n\le l+k_0$. 
  Now we pick an $n$-tile $X^n$, an $l$-tile  $X^l$, and an $(l+k_0)$-tile $X^{n+k_0}$ that all contain $x$. 
 Then  
 \begin{align*}
   \sigma(x,z) & \asymp \diam_\sigma(X^{l}) &&\text{ by Proposition~\ref{prop:chordalm} \ref{item:sxydiamX},} 
   \\
   & \asymp \diam_\sigma(X^{l+k_0})
   &&\text{ by 
   Proposition~\ref{prop:chordalm}~\ref{item:diamXkXkm1} 
   and~\ref{item:diamXkXkm2},}
   \\
   & \lesssim \diam_\sigma(X^n) \asymp \sigma(x,y) 
   &&\text{ by Proposition~\ref{prop:chordalm}~\ref{item:diamXkXkm1} and \ref{item:sxydiamX}.} 
 \end{align*}
In the previous estimates, 
all implicit  constants  can be chosen  independently  of $x,y,z$. Hence  the map 
 $\id_{\CDach}\colon (\CDach, \varrho)\ra (\CDach, \sigma)$ is indeed weakly
 quasisymmetric. The statement follows.
\end{proof}

We are now ready to prove  the second  part of Theorem~\ref{thm:S2vsf}.

\begin{proof}[Proof of Theorem~\ref{thm:S2vsf}~\ref{item:S2qsphere}]
  Let   $f\: S^2\ra S^2$ be an expanding Thurston map, and $\varrho$ be a visual metric for $f$. 
  
Suppose first that $f$ is
  topologically conjugate to a rational map.  Obviously, this 
    rational map is then 
  itself an expanding  Thurston map. Moreover, by 
Proposition~\ref{prop:conjisom} the conjugating homeomorphism  
 is  a snowflake equivalence with respect to visual
  metrics, and in particular a quasisymmetry.    
 This implies that in order to show that  the visual sphere 
 of $f$ is a quasisphere, we may actually assume that $f$ itself is a
rational Thurston map  that is expanding. 

Then $f\: \CDach\ra \CDach$ has no periodic critical points by Proposition~\ref{prop:rationalexpch}. So by Lemma~\ref{lem:sigqsvis} the visual metric $\varrho$ is quasisymmetrically equivalent to the chordal metric  $\sigma$. Hence the identity map $\id_{\CDach}\: (\CDach, \varrho)\ra (\CDach, \sigma)$ is a quasisymmetry, and
so  $(\CDach, \varrho)$ is a quasisphere. 
 This proves the first implication of the theorem. 




For the converse direction suppose that $f\: S^2\ra S^2$ is an
expanding Thurston map, $\varrho$ is a visual metric for $f$ on $S^2$, and
that there exists a quasisymmetry $h\colon (S^2, \varrho)\ra  (\CDach,
\sigma)$. Since all visual metrics are snowflake and hence also
quasisymmetrically equivalent, we may also assume that $\varrho$ is a visual
metric for $f$ satisfying \eqref{simmetric} in
Theorem~\ref{thm:visexpfactors1}.   
 
 The map  $h^{-1}$ is also a quasisymmetry; so  $h$ and $h^{-1}$  are $\eta$-quasi\-symmetric for some distortion function $\eta$. 
 We consider the conjugate $g=h\circ f\circ h^{-1}\colon \CDach \to \CDach$  of $f$ by $h$.

We claim that the family of iterates $\{g^n\:n\in \N\}$ is 
uniformly quasiregular,\index{quasiregular!uniformly}\index{uniformly!quasiregular} 
i.e., 
each map $g^n$ is   $K$-quasiregular with $K$ independent of $n$ (the definition of a quasiregular map was given near the end of Section~\ref{sec:QCgeom}).   This is true, because with the metric $\varrho$ satisfying \eqref{simmetric}, the map
    $f$ is locally ``conformal'', and  so  the dilatation of $g^n=h\circ f^n\circ h^{-1}$ can be bounded  by the dilatations of $h$ and $h^{-1}$,  and hence by a constant  independent of $n$.
    
  To be more precise, let $n\in \N$, $u\in \CDach$,  and for small $\eps>0$ consider points $v,w\in\CDach$ with $\sigma(u,v)=\sigma(u,w)=\epsilon$.  Define $x=h^{-1}(u)$, $y=h^{-1}(v)$,  $z=h^{-1}(w)$. By Theorem~\ref{thm:visexpfactors1}~\ref{item:visexpfactors2} we have that if $\epsilon>0$ is sufficiently small (depending on $u$ and $n$), then    
  \begin{align*}
    \frac{\sigma(g^n (u), g^n( v))}{\sigma(g^n(u), g^n(w))}
    &=
      \frac{\sigma(h(f^n (x)) , h(f^n (y)))}{\sigma(h(f^n (x)),
      h(f^n (z)))}
    \\
    & \leq \eta\left( \frac{\varrho(f^n (x),  f^n (y))}{\varrho(f^n(x) , f^n (z))}\right)
      =  \eta\left(\frac{\varrho(x,y)}{\varrho(x,z)}\right)
    \\
    & = \eta\left( \frac{\varrho(h^{-1}(u), h^{-1} (v))}{\varrho(h^{-1}(u),h^{-1}( w))}\right)\le H\coloneqq \eta(\eta(1)).
  \end{align*}
Hence 
\begin{align*}
  H(g^n,u) &\coloneqq
    \limsup_{\epsilon\to 0} \max
    \biggl\{ \frac{\sigma(g^n(u), g^n(v))}  
    {\sigma(  g^n (u),  g^n (w))}: 
    v,w\in \CDach,\, \sigma(u,v)=\sigma(u,w)=\eps\biggr\} 
  \\
   & \leq H 
  \end{align*}
  for all $u\in \CDach$ and $n\in \N$.
  This inequality implies that $g^n$ is locally
  $H$-quasiconformal on the set  $\CDach\setminus \crit(g^n)$ (according to the so-called ``metric'' definition of quasiconformality; see \cite[Section~34]{Va}).
 
 In particular, this shows that $g^n|\CDach\setminus \crit(g^n)$ is $K$-quasiregular with $K=K(H)$ independent of $n$. Since the finite set $\crit(g^n)$ is removable for quasiregularity (see \cite[Section 7.1]{Ri}), we conclude that $g^n$ is $K$-quasiregular with $K$ independent of $n$.  
 
 So the family of iterates $\{g^n\: n\in \N\}$ of $g$   is 
 uniformly quasiregular. 
  This implies that 
 $g$ is topologically conjugate to a rational map (see 
 \cite[Theorem~21.5.2]{IM}). 
Hence $f$ is also topologically conjugate to a rational map.  
\end{proof}
  
%

\begin{proof} [Proof of Theorem~\ref{thm:main01}]
By Proposition~\ref{prop:rationalexpch} the map $f$ is an expanding Thurston map (and hence also every iterate of $f$). So by   Theorem~\ref{thm:main}  for each large $n\in \N$ there exists a Jordan curve $\CC\sub \CDach$ with $\post(f)\sub \CC$ that is $f^n$-invariant. Equipped with the chordal metric, every such curve is a quasicircle as follows from Theorem~\ref{thm:Cquasicircle} and Lemma~\ref{lem:sigqsvis} (note that the class of visual metrics for $f$ and for any iterate of $f$ are the same). Statements 
 \ref{item:frat1_qc_ex} and \ref{item:frat_invC_qc} follow. 
 
Suppose the curve $\CC$ is actually $f$-invariant, and consider cells for $(f,\CC)$. Then by Proposition~\ref{prop:arc} the  family of all edges 
consists of uniform quasiarcs if the underlying metric on $\CDach$ 
is a visual metric for $f$. Again by  Lemma~\ref{lem:sigqsvis} we can switch to the chordal metric $\sigma$ in this statement. 

Similarly, the 
family of the boundaries  of all tiles consists of uniform quasicircles for $\sigma$.  Now it is a standard fact that a closed Jordan region $X\sub \CDach$ 
bounded by a quasicircle $\partial X$ is a quasidisk.
More precisely, if $h\: \partial \D\ra \partial X$ is an $\eta$-quasisymmetry, 
then it can be extended to an $\tilde\eta$-quasisymmetry $H\:\overline \D
\ra X$. Here $\tilde\eta$ depends not only on $\eta$, but also on a lower bound for the diameter of $\CDach\setminus X$ (see \cite[Proposition~5.3~(ii)]{Bo2}). 
In particular, if we have a family of such Jordan regions whose boundaries form a family of uniform quasicircles, then the family will consist of uniform 
quasidisks, if there is a positive  uniform lower bound for the  diameters of the complements of the regions. Since $f$ is expanding, there are only finitely many  tiles not contained in hemispheres, and so we have such a uniform lower bound for the family of all tiles for $(f,\CC)$. Statement~\ref{item:frat_Markov} follows.
  \end{proof}
  
   \begin{ex}\label{ex:ratnotenf} The ``if'' part of  Theorem~\ref{thm:S2vsf}~\ref{item:S2qsphere} is not true if one only requires that $f$ is Thurston equivalent to a rational map. 
   Indeed, in  Example~\ref{ex:barycentric} we considered two Thurston maps $f_2$
   and 
 $\widetilde{f}_2$ on a triangular pillow identified with the Riemann sphere 
  $\CDach$. Here  $f_2\: \CDach \ra \CDach$
is  rational and not expanding (it has a critical fixed point), while 
 $\widetilde{f}_2\: \CDach \ra \CDach$ is  an expanding (non-rational) Thurston map.
 Both maps realize the barycentric subdivision rule and are hence Thurston equivalent (Proposition~\ref{prop:rulemapex}). So 
  $\widetilde{f}_2$ is Thurston equivalent to a rational map. 
  
 Now  $\widetilde{f}_2$ also has a critical fixed point. So if 
 ${\varrho}$
  is  a visual metric for $\widetilde{f}_2$, then  $(\CDach,{\varrho})$
  is not doubling by Theorem~\ref{thm:S2vsf}~\ref{item:S2doubling}. On the other hand, each quasisphere is doubling, because $(\CDach, \sigma)$ is doubling and this condition is invariant under quasisymmetries. Hence 
  $(\CDach,{\varrho})$ cannot be a quasisphere. \end{ex}

\ifthenelse{\boolean{singlechapter}}{

%


\chapter[Rational expanding Thurston maps]{Rational Thurston maps and Lebesgue measure}

\label{cha:rati-thurst-maps-1}

 In this chapter we consider \emph{rational} expanding
Thurston maps  $f$ on the Riemann sphere $\CDach$. We will 
 investigate various measures that are associated with $f$. 
This will allow us to complete  the proof of Theorem~\ref{thm:S2vsf} by providing the missing justification 
for  part \ref{item:S2Lattes} of this theorem (see the end of Section~\ref{sec:lyapunov-exponent-mu}).
Theorem~\ref{thm:abscontimpLattes} below will be crucial for the characterization of 
Latt\`es maps 
given in Chapter~\ref{cha:latt-maps-comb}.

We equip the Riemann sphere  $\CDach$ with (normalized) Lebesgue measure $\leb_{\CDach}$ 
given by 
\index{LAA@$\leb$}
\index{Lebesgue!measure}
\index{measure!Lebesgue}
\begin{equation*}
   d\leb_{\CDach}(z)= \frac{1}{\pi (1+|z|^2)^2}\, d\leb_{\C}(z).  
\end{equation*}
The normalization means  that $\leb_{\CDach}(\CDach)=1$.
Mostly, we will drop the subscript and simply write $\leb=\leb_{\CDach}$ if  no ambiguity can arise. Throughout this chapter we will use ``Polish notation'' and denote by $|z-w|$ the chordal distance $\sigma(z,w)$ of two points $z,w\in \CDach$. All metric  notions will refer to the chordal metric unless otherwise 
mentioned.

A first link to measure-theoretic dynamics and ergodic theory is provided by the following statement.

\begin{theorem}
  \label{thm:ergodic_for_f}
  \index{Thurston map!rational}
  \index{ergodic}
  \index{measure!ergodic}
  Let $f\colon \CDach\to \CDach$ be a rational expanding
  Thurston map. Then Lebesgue measure $\leb_{\CDach}$ is ergodic for $f$. 
\end{theorem}
Note that $\leb_{\CDach}$ is essentially never $f$-invariant, but  ergodicity is interpreted as for $f$-invariant measures (see the discussion in Section~\ref{sec:reviewmdyn}): if $A\sub \CDach$ is a Borel set with $f^{-1}(A)=A$, then  
$\leb_{\CDach}(A)=0$ or  
 $\leb_{\CDach}(A)=1$.

Theorem~\ref{thm:ergodic_for_f} is well known. We will present a proof in Section~\ref{sec:ergod-lebesg-meas}, where  we follow the argument from
\cite[Theorem~3.9]{McM} closely.

In our context one can actually find an $f$-invariant measure that is absolutely continuous with respect to Lebesgue measure.

\begin{theorem}
  \label{thm:ex_inv_abs_L}
  \index{laa@$\lambda_f$}
  \index{Thurston map!rational}
  \index{invariant!measure}
  \index{measure!invariant}    
  \index{f-invariant@$f$-invariant!measure}
  \index{measure!absolutely continuous}
  \index{absolutely continuous measure}
  \index{ergodic}
  \index{measure!ergodic}
  \index{Lebesgue!measure}
  \index{measure!Lebesgue}
  Let $f\colon \CDach\to \CDach$ be a rational expanding
  Thurston map. Then there exists a unique $f$-invariant   (Borel) probability measure $\lambda_f$ on $\CDach$ that is absolutely
  continuous with respect to  Lebesgue measure $\leb_{\CDach}$. This measure has the form $d\lambda_f=\rho\, d\leb_{\CDach}$, where $\rho$ is a positive continuous function on $\CDach\setminus \post(f)$. Moreover, the 
   measure $\lambda_f$ is ergodic for $f$. 
\end{theorem}

It immediately follows from the second part of this  statement that 
$\lambda_f$ and $\leb_{\CDach}$ are mutually absolutely continuous and are hence in the same {\em measure class}, i.e., the two measures have precisely the same (Borel) null-sets.  The points    in $\post(f)$  are  singularities for 
the Radon-Nikodym derivative $\rho=d\lambda_f/d\leb_{\CDach}$. One can 
describe the asymptotic behavior near these points explicitly (see Proposition~\ref{rem:asymbevrho}).  

Theorem~\ref{thm:ex_inv_abs_L}  is again a well-known statement. 
 The  first part  is actually true for more general rational  maps
(see, for example, \cite[Theorem~3]{GPS}). 
We will present the proof of Theorem~\ref{thm:ex_inv_abs_L}  in Section~\ref{sec:Ruelle}. There we will reinterpret the existence problem for the measure $\lambda_f$ as a fixed point problem for a certain operator, the {\em Ruelle} or {\em transfer operator} acting on the space of continuous functions on $\CDach\setminus 
\post(f)$. 

There is  another natural measure for $f$  that is in the
same measure class as $\leb_{\CDach}$, namely  the 
\emph{canonical orbifold measure}\index{canonical orbifold!measure}\index{measure!canonical orbifold}\index{orbifold!canonical measure}\index{OAA@$\Omega,\Omega_f$}\index{push-forward!of measure!by orbifold covering map} 
$\Omega=\Omega_f$ that is associated 
with the orbifold $\OC_f=(\CDach, \alpha_f)$ of  $f$. To quickly
review its definition (see \eqref{eq:defOm} in Section~\ref{sec:expratThmaps} for more details), let $\Theta\colon X \to \CDach$ be the universal orbifold covering map of 
$\OC_f$.  Here  $X=\D$ or $X=\C$ depending on whether $\OC_f$ is hyperbolic or  parabolic. Roughly speaking, $\Omega$ is then the ``local'' push-forward of
the natural measure $\leb_X$ on $X$, namely  hyperbolic area in the hyperbolic and 
 Euclidean area in the parabolic case. More precisely, $\Omega$ is the unique measure on $\CDach$ such that for the Jacobian $J_{\Theta, \leb_X, \Om}$ of $\Theta$ with respect to 
$\leb_X$ and $\Om$ we have $J_{\Theta, \leb_X, \Om}= 1$ on $X$ (see Section~\ref{sec:jacobian-mu} for a general discussion of Jacobians). 
In the hyperbolic case, $\Om$  is independent of the choice of $\Theta$ and hence unique, but in the parabolic case $\Om$ is only unique up to a positive multiplicative constant.
One can enforce uniqueness in the parabolic case  by the normalization  $\Om(\CDach)=1$. 

A  Latt\`{e}s map $f$ has  a parabolic orbifold $\mathcal{O}_f$. In this case,  the  normalized  measure $\Om=\Om_f$ is equal to the measure of maximal entropy and the measure in Theorem~\ref{thm:ex_inv_abs_L}. 
  
\begin{theorem}
  \label{thm:Lattes_can_orb_meas}
  \index{Latt\`{e}s map}
  \index{measure!of maximal entropy}
  \index{n@$\nu_f$}
  \index{parabolic!orbifold}
  \index{orbifold!parabolic}
  \index{Thurston map!parabolic}
  \index{Lebesgue!measure}
  \index{measure!Lebesgue}
  \index{LAA@$\leb$}
  Let  $f\colon \CDach\to \CDach$ be a Latt\`{e}s map. Suppose $\nu_f$ is its measure of maximal entropy, $\lambda_f$ the unique $f$-invariant probability measure that is absolutely 
  continuous with respect to $\leb_{\CDach}$, and $\Om_f$ the canonical orbifold measure of $\OC_f$ normalized such that $\Om_f(\CDach)=1$. Then 
   \begin{equation*}
    \nu_f = \lambda_f=\Om_f. 
  \end{equation*}
\end{theorem}

An immediate consequence of the previous theorem is that for a Latt\`{e}s map
the measure of maximal entropy $\nu_f$ is absolutely continuous with respect to 
$\leb_{\CDach}$. This property actually characterizes these maps among rational expanding Thurston maps.

\begin{theorem}
  \label{thm:abscontimpLattes} 
  \index{Latt\`{e}s map} 
  \index{measure!of maximal entropy}
  \index{n@$\nu_f$} 
  \index{parabolic!orbifold}
  \index{orbifold!parabolic} 
  \index{Thurston map!parabolic} 
  \index{Lebesgue!measure}
  \index{measure!Lebesgue}
  \index{LAA@$\leb$}
  Let $f\: \CDach \ra \CDach$ be a rational expanding Thurston
  map.  Then its measure of maximal entropy $\nu_f$ is
  absolutely continuous with respect to Lebesgue measure
  $\leb_{\CDach}$ if and only if $f$ is a Latt\`es map.
\end{theorem} 

A much stronger version of this theorem is actually true. 
Namely, according to a result by A.\ Zdunik \cite{Zd}  Latt\`es maps are characterized 
by this property  among {\em all} rational maps, and not only among  
rational expanding Thurston maps.

This chapter is organized as follows. In Section~\ref{sec:jacobian-mu} we review some general facts about Jacobians. 
Section~\ref{sec:ergod-lebesg-meas} is devoted to the proof of Theorem~\ref{thm:ergodic_for_f} while
Theorems~\ref{thm:ex_inv_abs_L} and~\ref{thm:Lattes_can_orb_meas}
are established in Section~\ref{sec:Ruelle}. 
None of this material is new. We included it to make our presentation more self-contained. 
In
Section~\ref{sec:lyapunov-exponent-mu} we give a characterization of Latt\`es maps (see Theorem~\ref{thm:Zdunik}) that will lead to the proofs of 
Theorem~\ref{thm:abscontimpLattes} and  
Theorem~\ref{thm:S2vsf}~\ref{item:S2Lattes}.

\section{The Jacobian of a measurable  map}
\label{sec:jacobian-mu}

In this section we discuss some general facts about
Jacobians. We are mostly interested in Jacobians of holomorphic
maps on the Riemann sphere, but we will discuss the subject in greater generality.
For more background see
\cite[Section~2.9]{PU}.
 
Recall that a {\em measure space} is a triple
$(X,\mathcal{F}, \mu)$  consisting of  a set $X$, a $\sigma$-algebra
$\mathcal{F}$ on $X$, and a measure $\mu$ defined on the sets in
$\mathcal{F}$. Let $(X',\mathcal{F}', \mu')$ be another measure
space, and $T\: X\ra X'$ be a measurable map, i.e., $T$
satisfies $T^{-1}(A)\in \mathcal{F}$ whenever
$A\in \mathcal{F}'$.  In this context we call a set $A\sub X$
{\em admissible} if $A\in \mathcal{F}$, $T(A)\in \mathcal{F}'$,
and $T$ is injective on $A$.  Then a measurable function
$J_T\colon X\to [0,\infty]$ is called a
\emph{Jacobian}\index{Jacobian, $J_f$}
of $T$, if for each admissible
set $A\sub X$ we have
  \begin{equation*}
    \mu'(T(A))= \int_A J_T\,d\mu.
  \end{equation*}

We will sometimes write $J_T=J_{T,\mu,\mu'}$ if we 
also want to mention   the measures involved. Often 
our measure spaces will be identical, i.e., $(X,\mathcal{F},
\mu)=(X',\mathcal{F}', \mu')$. Then we  write
$J_T=J_\mu=J_{T,\mu}$ based  on which  dependence we want to
emphasize.

If $A\sub X$ is an admissible set, we can define a measure $\nu_A$ on $A$ given as 
$$\nu_A(M)=\mu'(T(M)), \text{ whenever $M\sub A$  is admissible.}$$
A   sufficient condition for the existence of a 
Jacobian $J_T$ is that there exists a countable partition  of $X$ into admissible sets
such that for each  set $A\sub X$ in the partition  
 the measure $\nu_A$ 
   is absolutely
continuous with respect to $\mu$. In this case,  $\nu_A$ is absolutely continuous 
with respect to $\mu$ for each admissible set $A\sub X$. Moreover,  
$J_T$ is uniquely determined 
 $\mu$-almost everywhere on $X$, because on each admissible set   
  $A$ it is equal to the  
Radon-Nikodym derivative $d\nu_A/d\mu$.

If $J_T$ is a Jacobian of $T$,  $A\sub X$ is an admissible set,   and
$\rho'\colon X'\to [0,\infty]$  a non-negative measurable function, then 
 \begin{equation}
  \label{eq:trafo_J}
  \int_{T(A)} \rho' \,d\mu'
  = 
  \int_{A} (\rho'\circ T) \cdot J_T\,d\mu. 
\end{equation}
A similar relation holds for all integrable functions $\rho'\in L^1(\mu')$.

Let $(X'', \mathcal{F}'', \mu'')$ be a third measure space, and
$S\colon X'\to X''$ be a measurable map. Then a chain rule for Jacobians is valid: 
if the Jacobians $J_T$ and $J_S$ exist,  
then a Jacobian of $S\circ T$ is given by  
\begin{equation*}
  J_{S\circ T}= (J_S\circ T)\cdot J_T. 
\end{equation*}

We will only  be  interested in the cases where the spaces 
 are open subsets of $\D,\C$, or $\CDach$ equipped 
 with their Borel $\sigma$-algebras,  and the measures are 
 absolutely continuous with respect to the corresponding  Lebesgue measures. In addition, 
$T=f$ will be a holomorphic
map. Then the existence of  Jacobians  follows
from the transformation formula for integrals. 

For example, let  $X=X'=\CDach$ and  $\leb=\leb_{\CDach}$ be  normalized Lebesgue measure on $\CDach$. Suppose $\mu$ and $\mu'$ are Borel measures on $\CDach$ that are absolutely continuous  with respect to $\leb$. Then 
 $d\mu= \kappa\, d\leb$ and $d\mu'= \kappa'\,d\leb$ for some non-negative 
 measurable functions 
 $\kappa$ and $\kappa'$ on $\CDach$. Let us assume in addition that $\kappa$ is positive 
  $\leb$-almost everywhere on $\CDach$. Then if $f\:\CDach\ra \CDach $ is a 
  rational  function, we have
\begin{equation}
  \label{eq:formula_J}
  J_{f,\mu, \mu'}(z)
  = 
  \frac{\kappa'(f(z))}{\kappa(z)}f^\sharp(z)^2
\end{equation}
for $\leb$-almost every $z\in \C$, where  
\index{spherical!derivative}
\index{$f^{\sharp}$}
\begin{equation*}
  f^\sharp(z)= \frac{1+|z|^2}{1+|f(z)|^2}|f'(z)|
\end{equation*}
is the spherical derivative of $f$. In particular, if $\mu=\mu'=\leb$, then 
\begin{equation}
  \label{eq:formula_J2}
  J_{f}=J_{f,\leb}=(f^\sharp)^2.
 \end{equation}

Now suppose in addition  that $f\colon \CDach \to \CDach$ is a rational Thurston map,  and  that $\mu=\mu'$ is an $f$-invariant
probability measure on $\CDach$.  We pick a 
 Jordan curve $\CC\subset \CDach$ with
$\post(f)\subset\CC$ such that $\leb(\CC)=\mu(\CC)=0$,  and
consider tiles for $(f,\CC)$. 
Since $f^n$ is a rational map, 
it preserves 
sets of $\leb$-measure zero. Hence $\leb(f^{-n}(\CC))=0$, which implies that $\mu(f^{-n}(\CC))=0$ for all $n\in \N_0$, because by our assumptions 
$\mu$ is absolutely continuous with respect to $\leb$. It follows that  all edges and hence all boundaries of tiles for $(f,\CC)$ are sets of $\mu$-measure zero. 

For each  $1$-tile $X\in
\X^1$ the map $f|X$ is a homeomorphism onto its image.   This shows that $1$-tiles are admissible and it follows that 
\begin{align}
  \label{eq:Jdmu}
  \int_{\CDach}  J_{f,\mu}\,d\mu&=\sum_{X\in\X^1}\int_X J_{f,\mu}\, d\mu= \sum_{X\in \X^1} \mu(f(X))\\
  &= \deg(f) (\mu(X^0_{\tt  b})+ \mu(X^0_{\tt  w})) = \deg(f). \notag
\end{align}

The measure-theoretic entropy $h_\mu(f) $ of $\mu$ for the given map $f$
can be
expressed using the Jacobian as 
\begin{equation}
  \label{eq:Rohlinsform}
  h_\mu(f)  = \int _{\CDach} \log (J_{f,\mu}) \,d\mu. 
\end{equation}
This is known as 
\emph{Rokhlin's formula}\index{Rokhlin's formula}\index{entropy!measure-theoretic}\index{measure-theoretic entropy}\index{h mu@$h_\mu$}
(see, for example, \cite[Theorem~2.9.7]{PU}).

If we combine \eqref{eq:Jdmu} and  \eqref{eq:Rohlinsform} with 
Jensen's inequality,\index{Jensen's inequality}
we recover the inequality  
$$ h_\mu(f) = \int_{\CDach} \log (J_{f,\mu}) \,d\mu\le \log \biggl (
 \int_{\CDach} J_{f,\mu} \,d\mu\biggr) =\log (\deg(f))$$
 that we derived in Chapter~\ref{cha:measure} for an arbitrary expanding Thurston map $f$.

\section{Ergodicity of Lebesgue measure}
\label{sec:ergod-lebesg-meas}

In this section we will prove Theorem~\ref{thm:ergodic_for_f}. 
To prepare the proof, we 
fix a rational expanding Thurston map $f\: \CDach \ra \CDach$
and define $\V^\infty= \bigcup_{n\in \N_0}f^{-n}(\post(f))$.
We require a lemma. 

\begin{lemma}
  \label{lem:points_away_post}
  There exists a neighborhood $U\subset \CDach$ of $\post(f)$ with the following property:
 if $z_0\in \CDach \setminus \V^\infty$ is arbitrary and if we define $z_n=f^n(z_0)$ for $n\in \N_0$, then there exists a subsequence $\{ z_{n_k}\} $ of $\{z_n\}$ such that  
  $z_{n_k}\in \CDach \setminus U$ for all $k\in \N$.
\end{lemma}

\begin{proof}
For the proof it  is convenient to use a visual metric $\varrho$ for $f$ as provided by
  Theorem~\ref{thm:visexpfactors1}~\ref{item:visexpfactors2}. Let $\Lambda>1$ be 
the  expansion factor of $\varrho$. If we pick   $k_0\in
  \N$ sufficiently large, then for each  $p\in \post(f)$ the ball
  $U_p\coloneqq B_{\varrho}(p, \Lambda^{-k_0})$ satisfies
\eqref{simmetric}, i.e.,
\begin{equation*}
  \varrho(f(x), f(p)) = \Lambda\varrho(x,p) 
  \text{ for all } x\in U_p,
 \end{equation*}
  and we also have
 \begin{equation*}\varrho(p,q) \geq 2\Lambda^{-k_0 + 1}
  \text{ for distinct points } p,q\in \post(f). 
\end{equation*}
In particular, the balls $U_p$ and $U_q$ are disjoint for
distinct $p,q\in \post(f)$. Now define 
$U= \bigcup_{p\in\post(f)}U_p$. This is an open neighborhood of the set $\post(f)$.

To see that $U$  has the property as in the statement,  
let $z_0\in\CDach \setminus
\V^{\infty}$ be arbitrary and consider  $z_n=f^n(z_0)$ for some $n\in
\N$. It suffices to show that  there exists  $m\geq n$ such that
$z_m=f^m(z_0)\in \CDach\setminus U$. 

If $z_n\in \CDach \setminus U$ we
are done. So assume  $z_n \in U$; then $z_n\in U_p$ for a point 
$p\in \post(f)$. Since $z_0\notin \V^\infty$, we have $z_n\notin \post(f)$ and so $z_n\ne p$. On the other hand, $z_n\in U_p=B_{\varrho}(p, \Lambda^{-k_0})$, and so there exists 
a number $k\in \N$, $k\ge k_0$ such that 
$\Lambda^{-k-1} \leq \varrho(z_n,p) < \Lambda^{-k}$. We now consider two cases.  

\smallskip
{\emph{Case 1:}} We have $k=k_0$ and so 
$\Lambda^{-k_0 -1} \leq \varrho(z_n,p) < \Lambda^{-k_0}$. Then $ \Lambda^{-k_0} \leq \varrho(f(z_n), f(p)) < \Lambda^{-k_0+1}$. 
So if  $q=f(p)$, then $q\in  \post(f)$ and $z_{n+1}= f(z_n)\notin U_q$. Moreover,
 for each  $p'\in \post(f)\setminus \{q\}$
we have 
\begin{equation*}
  \varrho(z_{n+1}, p') 
  \geq 
  \varrho(p', q) - \varrho(z_{n+1}, q) 
  \geq
  2\Lambda^{-k_0 +1} - \Lambda^{-k_0 +1} 
  \geq 
  \Lambda^{-k_0},
\end{equation*}
and so  $z_{n+1} \notin U_{p'}$. Thus $z_{n+1}\notin U$ as
desired.

\smallskip
{\emph{Case 2:}}
 $ \Lambda^{-k -1} \leq \varrho(z_n,p) < \Lambda^{-k}$ for some
 $k> k_0$. In this case,
$z_{n+k-k_0}=f^{k-k_0}(z_n)$ satisfies the
assumptions of Case~1 for the point 
$p'=f^{k-k_0}(p)\in \post(f)$, and so   $z_{n+k-k_0 +1} \in \CDach \setminus U$
as desired.
\end{proof}

In the proof of 
Theorem~\ref{thm:ergodic_for_f} we need 
the  
\emph{Lebesgue density theorem}.\index{Lebesgue!density theorem}\index{Lebesgue!measure}\index{measure!Lebesgue}\index{LAA@$\leb$}
For Lebesgue measure $\leb$ on $\CDach$ this theorem says that if $A\subset \CDach$ is a Borel
set, then $\leb$-almost every point $z_0\in A$
is a \emph{(Lebesgue) density point} of $A$, meaning that
\begin{equation}
  \label{eq:L_density}
  \lim_{\epsilon\to 0^+} 
  \frac{\leb(B(z_0,
    \epsilon)\cap A)}{\leb(B(z_0,\epsilon))} 
  =1. 
\end{equation}
Here and in the following balls are defined  with respect to the chordal metric on 
$\CDach$.  We need a variant  of \eqref{eq:L_density} where we allow 
``roundish'' sets with controlled ``eccentricity'' instead of metric balls.  

\begin{lemma}
  \label{lem:L_density_qballs}
  Let $A\subset \CDach$ be a Borel set and $z_0\in A$ be a
   density point of $A$. 
   Suppose that $V_n\sub \CDach$ is a Borel set for  $n\in \N$ such
   that  
  \begin{equation*}
    B(z_0, r_n/K) \subset V_n \subset B(z_0, r_n),
  \end{equation*}
  where $K\ge 1$, $r_n>0$,  and $r_n\to 0$ as $n\to \infty$.
  Then
  \begin{equation*}
    \lim_{n\to \infty} \frac{\leb(V_n\cap A)}{\leb(V_n)} =1. 
  \end{equation*}
\end{lemma}

\begin{proof}
  Note that \eqref{eq:L_density} is equivalent to
  \begin{equation*}
    \frac{\leb(B(z_0, \epsilon) \setminus A)}{\leb(B(z_0,
      \epsilon))} \to 0
    \text{ as } \epsilon\to 0^+.
  \end{equation*}
  Suppose the sets $V_n$, $n\in \N$, are as  in the statement. 
  Define  $B_n= B(z_0, r_n)$ and $B'_n=B(z_n, r_n/K)$ for $n\in \N$.  Then
  \begin{equation*}
    \leb(V_n\setminus A) \leq \leb(B_n\setminus A)
  \end{equation*}
 and   \begin{equation*}
    \leb(V_n) \geq \leb(B'_n) \gtrsim  \leb(B_n),
  \end{equation*}
  where $C(\gtrsim)$ is independent of $n$. Hence 
  \begin{equation*}
    \frac{\leb(V_n\setminus A)}{\leb(V_n)} \lesssim 
    \frac{\leb(B_n\setminus A)}{\leb(B_n)} \to 0
  \end{equation*}
  as $n\to \infty$, and the statement  follows. 
\end{proof}
After these preparations we are ready to prove the main result of this section. 

\begin{proof}[Proof of Theorem~\ref{thm:ergodic_for_f}]
Suppose  $A\subset \CDach$ is 
  a Borel set with $f^{-1}(A) =A$ and   $\leb(A)>0$.  We have to show that $A$ has full measure, i.e.,   $\leb(A)=1$. 
  Since $f$ is surjective, we also have $f(A)=A$, and so $A$ is {\em fully invariant} under 
$f$ in the sense that $f^{-1}(A) =A=f(A)$. 

Before we delve into the details, let  us give an outline of the argument.
By the  Lebesgue density theorem, we can find  a sequence of small
balls where $A$ has  density approaching $1$. We then use the dynamics and  blow up these balls by iterates of $f$ to find a sequence of balls of fixed  size with this property. Passing to a subsequential limit, we    obtain a ball $B$ where 
$A$ has density $1$. Since
$f$ is eventually onto and $A$ is fully invariant,  we conclude that 
$A$ must have  full measure. We now present the  details. 

\smallskip
{\em Claim 1.} There exists $\delta>0$,  and  balls
  $B_k=B(w_k,\delta) \sub \CDach$ for $k\in \N$ such that 
$\leb(B_k\setminus A) \to 0$   as  $k\to \infty$.

\smallskip 
The main point here is that the balls $B_k$ have a fixed radius.
To prove the claim, first note that   $\V^\infty= \bigcup_{n\in \N_0}f^{-n}(\post(f))$ is a countable set, $A$ has positive measure, and so we can find a  Lebesgue density point 
  $z_0\in A\setminus \V^\infty$ of $A$.
    Now let $U\subset \CDach$ be a neighborhood of $\post(f)$ as
  provided by  Lemma~\ref{lem:points_away_post}. Then there is a
  subsequence $\{z_{n_k}\}$ of the sequence $\{z_n\}$ given by  $z_n=
  f^n(z_0)$ for $n\in \N_0$ such that $z_{n_k} \in \CDach \setminus U$ for all
  $k\in \N$. 

  Let $\delta_0\coloneqq  \dist(\post(f), \CDach\setminus U)>0$. Then the disk 
  $B_k'\coloneqq B(z_{n_k}, \delta_0)$ is a simply
  connected region  contained in $\CDach \setminus \post(f)$. In particular, each iterate is a covering map over $B_k'$ and so there exists a conformal map  $g_k$ on $B_k'$
  that is an inverse branch 
  of $f^{-n_k}$
   and sends  $z_{n_k}$ to
  $z_0$. 
  
  Let $B_k\coloneqq  B(z_{n_k}, \delta_0/2)$ and $V_k\coloneqq 
  g_k(B_k)$. 
We note that the diameters of the sets $V_k'\coloneqq g_k(B'_k)
\supset V_k$ tend
to $0$  as $k\to
  \infty$. The quickest way to see this is to use the canonical orbifold
  metric $\omega$ of $f$. Namely, it follows from 
  Proposition~\ref{lem:R_orbimetric} that there  is a constant
  $\rho>1$ such that
  \begin{equation*}
    \diam_\omega(V_k') \lesssim \rho^{-n_k}
  \end{equation*}
  for  $k\in \N$ with $C(\lesssim)$ independent of $k$. 
  Since $\omega$ induces the standard topology 
  on $\CDach$, 
it follows that for the chordal metric we have $\diam_\sigma(V'_k) \to 0$ (and hence also  $\diam_\sigma(V_k) \to 0$) as $k\to \infty$.

 Koebe's  distortion theorem 
(see \eqref{eq:Koebe14} in Theorem~\ref{thm:Koebe}; note that $V'_k=g_k(B'_k)$ is contained in a hemisphere of $\CDach$ for large $k$)
 implies that  the sets $V_k$ satisfy the assumption of
  Lemma~\ref{lem:L_density_qballs} for 
the density  point $z_0=g_k(z_{n_k})\in V_k$ 
of $A$. Thus
  \begin{equation*}
    \frac{\leb(V_k\cap A)}{\leb(V_k)} \to 1
  \end{equation*}
  as $k\to \infty$. 
  Since $A$ is $f$-invariant, it again follows from Koebe's
  distortion theorem (more precisely, we apply  \eqref{eq:Koebe_dg} in Theorem~\ref{thm:Koebe}) that  
   \begin{align*}
    \leb(V_k \setminus A)
    &\asymp 
    \leb(B_k\setminus A) J_{g_k}(z_{n_k}) 
    \asymp
    \leb(B_k\setminus A)\frac{ \leb(V_k)}{\leb(B_k)}
     \asymp 
    \leb(B_k\setminus A) \leb(V_k).
  \end{align*}
  Here $J_{g_k}=(g_k^\sharp)^2$ is the Jacobian of $g_k$ with respect to $\leb$ (see \eqref{eq:formula_J2}) and the 
  constants $C(\asymp)$ are independent of $k$. Hence
  \begin{align*}
  \leb(B_k\setminus A) \asymp \frac{\leb(V_k \setminus A)}{\leb(V_k)}\to 0
  \end{align*}
  as $k\to \infty$.  Claim~1  follows (with $w_k=z_{n_k}$ and $\delta=\delta_0/2$).

  \smallskip
 {\em Claim 2.} There exists an open  ball $B\sub \CDach$ with
  $\leb(B\setminus A) = 0. $

 \smallskip 
  Indeed,  if $B_k=B(w_k, \delta)$, $k\in \N$, 
is a sequence of balls as in 
  Claim~1, 
then by compactness  $\{w_k\}$  has  a convergent subsequence 
  which we 
  still denote by $\{w_k\}$ for convenience. Let $w_\infty\coloneqq 
  \lim_{k\to \infty} w_{k}$ and $B\coloneqq  B(w_\infty,
  \delta/2)$. 
  
Then  $B\subset B_k$ for  sufficiently large $k$ and so 
  $\leb(B\setminus A) \leq \leb(B_k\setminus A)$.
    Since the right hand side tends  to $0$ as $k\to \infty$, we conclude 
  that $\leb(B \setminus A) = 0$. Claim~2 is proved.

\smallskip
  We know (see Lemma~\ref{lem:event_onto})  that $f$
  is eventually onto; this implies that there exists  $n\in \N$ such that
  $f^n(B) = \CDach$ if  $B$ is a ball as in the previous claim. Since $A$ is completely $f$-invariant, we then have  
  $\CDach \setminus A=f^n(B\setminus A)$. Now $f^n$, as a rational map, preserves sets of measure zero, and so $\leb(\CDach\setminus A)=0$. This implies that $\leb(A) =
 1$. Thus $\leb$  is ergodic for $f$.
\end{proof}

We record a quick consequence of Theorem~\ref{thm:ergodic_for_f}.

\begin{cor} \label{cor:ergod} Let $f\: \CDach\ra \CDach$ be a rational expanding Thurston map,  and $\mu$ an 
$f$-invariant probability measure that is absolutely continuous with respect to $\leb$.
Then $\mu$ is ergodic for $f$.
\end{cor}

\begin{proof} Let $A\sub \CDach$ be a Borel set with $f^{-1}(A)=A$ and $\mu(A)>0$. We have to show that $\mu(A)=1$.

Now $\mu(A)>0$ and so $\leb(A)>0$ since $\mu$ is absolutely continuous with respect to 
$\leb$. Hence $\leb(A)=1$ as follows from  Theorem~\ref{thm:ergodic_for_f}, or equivalently
 $\leb (\CDach\setminus 
A)= 0$. This implies $\mu(\CDach\setminus A)=0$ which gives  $\mu(A)=1$ as desired.  
\end{proof}

\section{The absolutely continuous  invariant measure}
\label{sec:Ruelle}
 
 In this section  $f\: \CDach\ra \CDach$ is again a rational
Thurston map. We assume that $f$ is expanding, or equivalently, that $f$ has  no periodic critical points (see  Proposition~\ref{prop:rationalexpch}). Our goal is to 
construct an $f$-invariant probability measure $\lambda=\lambda_f$ on
$\CDach$ that is absolutely continuous with respect to (normalized) Lebesgue
measure $\leb=\leb_{\CDach}$ on $\CDach$.

In order to motivate our approach, let us consider 
a finite  (Borel)  measure $\mu$ on
$\CDach$ that is absolutely continuous with respect to $\leb$. 
Then $d\mu=\rho \,d\leb$ for a suitable  integrable  function $\rho$ on $\CDach$. As we will see momentarily, then 
$f_*\mu$ is also absolutely continuous with respect to $\leb$ and
hence can be written in the form $d(f_*\mu)=\mathcal{R}(\rho)\, d \leb$, where
$\mathcal{R}(\rho)$ is  integrable.  The   requirement that $\mu$ is $f$-invariant, i.e., the relation $f_*\mu=\mu$, then translates into the condition  $ \mathcal{R}(\rho)=\rho$. 

Let us  assume that $\rho$ has some additional regularity, namely, that it is a non-negative  continuous function outside the  finite subset 
$\post(f)$ of $\CDach$. 
If $p$ is a point in $\CDach\setminus \post(f)$, then $p$ is not
 a critical value of $f$ and so it has precisely $d=\deg(f)$ preimage points $q_1, \dots, q_d$. Near each of these points $q_j$ the map $f$ is a conformal map. In particular, there exists 
an open neighborhood $V$ of $p$, and pairwise disjoint  open neighborhoods
$U_j$ of $q_j$ such that $f|U_j$ is a conformal map  of $U_j$ onto 
$V$ for $j=1, \dots, d$. 

Let $g_j\coloneqq (f|U_j)^{-1}$. We will now consider Jacobians of these 
maps with $\leb$ as the underlying 
measure. Recall that if $h$ is a holomorphic map  defined on a subset  of $\CDach$, then  
$ J_h=J_{h,\leb} =(h^\sharp)^2, $
where $ h^\sharp$  denotes the spherical derivative (see  \eqref{eq:formula_J2}).  
For the Jacobians 
of $g_j$ and $f$  we have the 
relation 
\begin{equation}
  \label{eq:Jgj_Jf}
  J_{g_j}(w) = J_{f}(g_j(w))^{-1}
\end{equation}
for $w\in V$
which follows from the chain rule.  

Now consider a Borel set  $B\sub V$. Then its preimage under $f$ decomposes into the Borel sets $A_j=g_j(B)$, $j=1, \dots, d$.
 It follows that 
\begin{align}\label{eq:abscontcomp}
 f_*\mu(B)
  &= \mu (f^{-1}(B))= \sum_{j=1}^d \mu (A_j) 
   =\sum_{j=1}^d \int_{A_j}\rho\, d\leb\\
 & =\sum_{j=1}^d \int_{B} (\rho \circ g_j)\cdot  J_{g_j}\, d\leb \nonumber \\
&=  \int_{B} \sum_{j=1}^d  \rho( g_j(w)) \cdot J_{f}( g_j(w))^{-1}\, d\leb
(w) \nonumber \\&
=\int_{B} \sum_{z\in f^{-1}(w)}\rho(z)J_f(z)^{-1}\, d\leb(w) .\nonumber 
\end{align}
The equality between the first and the last terms in this identity actually remains valid for arbitrary Borel sets $B\sub \CDach\setminus
 \post(f)$, because we  can split each such set $B$ into countably many disjoint pieces such that each of these pieces lies  in a suitable 
set $V$ as chosen above. 

It follows that $f_*\mu$ is also absolutely continuous with respect to 
$\leb$ and has a Radon-Nikodym derivative equal to 
\begin{equation}\label{eq:Ruelle0}
\mathcal {R}  (\rho)(w)\coloneqq \sum_{z\in f^{-1}(w)}\rho(z)J_f(z)^{-1}
\end{equation} 
for  $w\in \CDach\setminus \post(f)$. Note that if $w\in \CDach\setminus  \post(f)$ and $z\in f^{-1}(w)$, then $z\in \CDach\setminus (\crit(f)\cup \post(f))$.
In particular,   $J_f=(f^\#)^2$ is a non-zero continuous function near such a point $z$. In \eqref{eq:abscontcomp} 
we saw that with the given notation we have the  local representation 
\begin{equation}\label{eq:localrepRuelle}
\mathcal {R}  (\rho)(w)=\sum_{j=1}^d  \rho( g_j(w)) \cdot J_{f}( g_j(w))^{-1}.
\end{equation}
Hence $\mathcal {R}  (\rho)$ is a non-negative continuous function on $\CDach\setminus  \post(f)$. In particular,
$f_*\mu=\mu$ if and only if $\mathcal{R}(\rho)=\rho$ on 
$\CDach\setminus \post(f)$. 

The idea for the construction of our desired measure $\lambda$ is to turn this consideration around. In the following, we use the notation $$\CDach_P=\CDach\setminus \post(f)$$ for the 
Riemann sphere ``punctured'' at the points in $\post(f)$, and denote by  
$C(\CDach_P)$  the 
space of real-valued continuous functions 
on
$\CDach_P$. We  introduce an operator, 
the {\em Ruelle}
or {\em transfer operator}\index{Ruelle operator}\index{transfer operator}\index{RAA@$\mathcal{R}$}
$\mathcal{R}$ that maps  each function 
$\rho\in C(\CDach_P)$ to the function $\mathcal{R}(\rho)$ as defined in \eqref{eq:Ruelle0}.
  Finding our measure $\lambda$ then amounts to finding a suitable fixed point of this operator. This is a more tractable  problem.

We will only use  some very basic properties of the Ruelle
operator summarized in the next lemma. For a thorough treatment
in a more general setting see for example \cite[Chapter 5]{PU}.

\begin{lemma} The Ruelle operator has the following properties.
  \label{lem:Rprops}
   \begin{enumerate}
   
    \item      \label{item:Ruelle1} If $\rho \in C(\CDach_P)$, 
    then $\mathcal{R}(\rho) \in C(\CDach_P)$. If, in addition,
     $\rho>0$ on $\CDach_P$, then $\mathcal{R}(\rho)>0$ on $\CDach_P$.
     
    \item      \label{item:Ruelle1a}   
 If $\rho \in 
  C(\CDach_P)\cap   L^1(\CDach)$, then $\mathcal{R}(\rho) \in
   C(\CDach_P)\cap   L^1(\CDach)$ and 
     \begin{equation*}
      \int_{\CDach} \rho \,d\leb
      = 
      \int_{\CDach} \mathcal{R}(\rho)\,d\leb. 
    \end{equation*}
    
       \item      \label{item:Ruelle1b} The operator $\mathcal{R}\: C(\CDach_P) \ra 
       C(\CDach_P)$ is linear,  and continuous in the following sense:  if $\rho$ and $\rho_n$ for $n\in \N$ are functions in $C(\CDach_P)$ such that $\rho_n\to \rho$ locally uniformly on 
       $\CDach_P$, then $\mathcal{R} (\rho_n)\to \mathcal{R} (\rho)$ locally uniformly on $\CDach_P$. 
  
  \item 
    \label{item:Ruelle2}
    If $\mu$ is a Borel measure on $\CDach$  of the form $d\mu =
    \rho\, d\leb$ 
    with 
    a function
   $\rho\in C(\CDach_P) \cap L^1(\CDach), $
  then  $d(f_*\mu )=\mathcal{R}(\rho)\, d\leb$. 
  %
  In particular, $\mu$ is 
    $f$-invariant if and only if $\mathcal{R}(\rho)=\rho$. 
  \end{enumerate}
\end{lemma}
Here $L^1(\CDach)$ denotes 
the space of real-valued functions 
that are almost everywhere defined on $\CDach$ and are integrable  
with respect to $\leb$. Note that for questions of integrability it is irrelevant whether a function is 
only defined on $\CDach_P$ or the whole Riemann sphere.

\begin{proof}  \ref{item:Ruelle1} If $\rho \in C(\CDach_P)$, 
then  $ \mathcal{R} (\rho) \in C(\CDach_P)$ as follows from the local representation 
of the function $ \mathcal{R} (\rho)$  in \eqref{eq:localrepRuelle}. Note that here it is important that if $w\in 
 \CDach_P$, then near each point $z\in f^{-1}(w)$, the function $J_f=(f^\#)^2$ does not vanish and is continuous. 
 If $\rho$ is positive, then it is clear that $\mathcal{R} (\rho)$ is also positive. 
 
 \smallskip 
  \ref{item:Ruelle1a}+\ref{item:Ruelle2} Let  $\rho\in C(\CDach_P)\cap L^1(\CDach)$, If, in addition,  
  $\rho\ge 0$, then   we consider the measure $\mu$ on $\CDach$ 
with  $d\mu=\rho\, d\leb$.  The computation in \eqref{eq:abscontcomp} shows that 
the measure $f_*\mu$ is absolutely continuous with respect to $\leb$ and $d(f_*\mu)=
 \mathcal{R} (\rho)\, d\leb$. In particular,   
 \begin{equation*}
   \int_{\CDach}  \mathcal{R} (\rho)\, d \leb
   =
   (f_*\mu)(\CDach)
   =
   \mu(f^{-1}(\CDach))
   =
   \mu(\CDach)
   =
   \int_{\CDach}  \rho\, d \leb. 
 \end{equation*}
 Thus $\mathcal{R} (\rho)\in 
 C(\CDach_P)\cap  L^1(\CDach)$.
 
  If  $\rho\in C(\CDach_P)\cap L^1(\CDach)$ is arbitrary, then we
  can split $\rho$ as $\rho=\rho_{+}- \rho_{-}$, where 
  $\rho_+, \rho_-\in C(\CDach_P)\cap L^1(\CDach)$, and 
  $\rho_+, \rho_-\ge0$. Obviously, $\mathcal{R} (\rho)
  =\mathcal{R} (\rho_+)-\mathcal{R} (\rho_-)$. 
 Statements \ref{item:Ruelle1a} and~\ref{item:Ruelle2} then 
 immediately follow.   
 
 \smallskip
 \ref{item:Ruelle1b} By what we have seen in \ref{item:Ruelle1}, we can consider the Ruelle operator as a map 
 $\mathcal{R} \:C(\CDach_P) \ra C(\CDach_P)$. It is clear that 
 $\mathcal{R}$ is a linear map.  If we equip  $C(\CDach_P)$  with the topology of locally uniform convergence,
 then  the continuity of $\mathcal{R}$ 
immediately follows from the local representation formula \eqref{eq:localrepRuelle}. Note that the inverse branches $g_j$ of $f$ 
 map a compact neighborhood $K\sub \CDach_P$ of the point $p\in \CDach_P$ near which these branches were defined again to compact subsets of  $\CDach_P$. 
\end{proof} 

Our goal now is to construct a suitable fixed point of $\mathcal{R}$. This is based on an iteration and averaging procedure. In order to be able to pass to  limits, we want to apply the Arzel\`a-Ascoli theorem for a certain subfamily  of $C(\CDach_P)$.
For this we will introduce a control function $M$ that will allow
us to establish  the desired equicontinuity and uniform boundedness properties. 

For the definition of $M$ we consider the functions $\rho_n=\mathcal{R}^n(1)$ for $n\in \N_0$. Here $1$ represents the function with constant value $1$ on $\CDach_P$, and $\mathcal{R}^n$ is the $n$-th iterate 
of $\mathcal{R}$ considered as an operator on $C(\CDach_P)$. We 
use the convention that $\mathcal{R}^0$ is the identity map on $C(\CDach_P)$ and so
$\rho_0= 1$. 

The definition of the Ruelle operator and the chain rule for Jacobians imply that 
$$ \rho_n(z)=\sum_{z'\in f^{-n}(z)}J_{f^n}(z')^{-1}$$ 
for $n\in \N_0$ and $z\in \CDach_P$.  
It  follows from Lemma~\ref{lem:Rprops}~\ref{item:Ruelle1} and \ref{item:Ruelle1a} that for $n\in \N_0$ the function $\rho_n$ is a  positive continuous function on $\CDach_P$ that is  normalized such that
\begin{equation}\label{eq:rhonnormal}
\int_{\CDach} \rho_n\, d \leb = \int_{\CDach} \mathcal{R}^n(1)\, d \leb
= \int_{\CDach} 1\, d \leb =1.\end{equation}

For  $z,w\in  \CDach_P$ we now define 
\begin{equation}\label{eq:defMzw}
  {M}(z,w)
  = 
  \sup_{n\in \N_0}\frac{\rho_n(z)}{\rho_n(w)}. 
\end{equation}
A priori it is not clear that $M$ is everywhere finite, but we will verify this momentarily. 
It immediately follows from the definition of $M$ that 
\begin{equation}\label{eq:Mcontrolsrhon} \rho_n(z)\le M(z,w)\rho_n(w)
\end{equation}
for all $n\in \N_0$ and $z,w\in \CDach_P$. This combined with the properties of $M$ as formulated in the next lemma will give us the desired control for the behavior of the functions 
$\rho_n$ (see the proof of Lemma~\ref{lem:fixed-point-widet}).

\begin{lemma} 
  \label{lem:defcontrolM}
  Let ${M}$ be defined as in \eqref{eq:defMzw}. Then the following statements are true:
  
  \begin{enumerate}
  \item  
    \label{item:defcontrolM0} 
    The function $M$ is continuous  on 
    $\CDach_P\times   \CDach_P$ and takes values in 
    $[1,\infty)$. 
    
  \item 
    \label{item:Mzw0_bounded}
    For each $w_0\in\CDach_P$ 
    the map $z\mapsto M(z,w_0) $  is integrable (with respect
    to $\leb$).   
  \end{enumerate}
\end{lemma}
  
\begin{proof}  \ref{item:defcontrolM0}
 It is clear that 
  ${M}(z,w) \ge \rho_0(z)/\rho_0(w)=1$ and so ${M}(z,w)\in [1,\infty]$ 
   for $z,w\in\CDach_P$. We also have $M(z,z)=1$
   for $z\in \CDach_P$.

Let $w\in\CDach_P$ be arbitrary. We first show that $M(z,w)\to 1$ as 
$z\to w$. To see this,  we fix   $\delta>0$ such that 
  $B'\coloneqq B(w, \delta)\sub \CDach_P $. Then for each   $n\in \N$ the map  $f^n$ is a covering map
  over $B'$. In particular, $f^n$ has $d^n$ inverse branches $g_1, \dots, g_{d^n}$, where 
  $d=\deg(f)$. Each map $g_j$ is a conformal map of $B'$ onto its image. As in the proof 
  of Theorem~\ref{thm:ergodic_for_f}, one sees that the diameters of these images tend
  to $0$ as $n\to \infty$. So with 
at most finitely many exceptions, 
these inverse branches of  iterates $f^n$, $n\in \N$,  send $B'$ to a set contained in a hemisphere. 
  So we can apply  Koebe's distortion theorem (see \eqref{eq:Koebe_dg} in Theorem~\ref{thm:Koebe} and the discussion after the proof of Theorem~\ref{thm:Koebe}) and conclude that for $z\in B\coloneqq B(w, \frac12 \delta)$, we have 
  \begin{equation}\label{eq:Koebebr}
   \frac{J_{g_j}(z)}{J_{g_j}(w)}= \frac{g_j^\sharp(z)^{2}}
   {g_j^\sharp(w)^{2}}\asymp1,
   \end{equation}
   where $C(\asymp)$ is independent of $n$ and the choice of the inverse branch $g_j$. 
   In addition, 
  $C(\asymp) \to 1$ as $z\to w$. 
  Based on  \eqref{eq:Jgj_Jf}, it follows that 
  \begin{align*}
  \rho_n(z)&=\sum_{z'\in f^{-n}(z)}J_{f^n}(z')^{-1}=\sum_{j=1}^{d^n}
 J_{f^n}(g_j(z))^{-1}= \sum_{j=1}^{d^n} J_{g_j}(z)\\
  &\le C(z) \sum_{j=1}^{d^n} J_{g_j}(w)= C(z)\rho_n(w), 
  \end{align*} 
  where $C(z)\to 1$ as $z\to w$. In particular,
  $$1\le M(z,w)\le C(z) $$
  and  so $M(z,w)\to 1$ as $z\to w$. A similar estimate based on \eqref{eq:Koebebr} also shows that $M(w,z)\to 1$ as $z\to w$.
  
We conclude that for each point 
$w\in \CDach_P$ there exists a neighborhood 
  $B$ of $w$ such that $M(z,w)$ is finite for $z\in B$. 
  From this and  the (obvious) inequality
  \begin{equation}
    \label{eq:MuwMuvMvw}
    {M}(u,w)\le {M}(u,v){M}(v,w)
  \end{equation}
  for $u,v,w\in  \CDach_P$ in combination with a chaining argument, the finiteness of $M(z,w)$ follows for all $z,w\in  \CDach_P$. 

Inequality \eqref{eq:MuwMuvMvw} implies that  
   $$ \frac1 {M(z_0,z)M(w,w_0)} \le                 \frac{M(z,w)}{M(z_0,w_0)}\le M(z,z_0)M(w_0,w)$$ 
   for 
  $ z,z_0,w,w_0\in \CDach_P$. By what we have seen, the first and the third expression in this inequality approach $1$ as $z\to z_0$ and $w\to w_0$. The continuity of $M$ easily follows. 
  
  \smallskip
   \ref{item:Mzw0_bounded} By the first part of the proof we know that for fixed $w_0\in \CDach _P$ the function $z\to M(z,w_0)$ is continuous on $\CDach_P$. In order
   to show its integrability, we need an integrable  upper bound for this function 
 near each  of the singularities in $\post(f)$. The following claim provides such a bound.
   
   \smallskip
   {\em Claim.} If $w_0\in  \CDach _P$ and $p\in \post(f)$, then
   there exists $\alpha\in (0,2)$ such that 
   \begin{equation*}
     M(z,w_0) 
     \lesssim 
     |z-p|^{-\alpha} \text{ for $z$ near $p$. }
   \end{equation*}

   \smallskip
 Once we know that the claim is true, the integrability of $z\to M(z,w_0)$
  follows, because for $\alpha<2$ the function $z\mapsto |z-p|^{-\alpha}$ is locally integrable near $p$.
  
 One can say more about the exponent $\alpha$ here.  Indeed, for $p\in \CDach$ we define 
  \begin{equation} \label{eq:defbetaf}
  \beta_f(p)=\max\{\deg(f^n, q): n\in\N \text{ and } f^n(q)=p\}.
   \end{equation}
Since  $f$ has no periodic critical points, the local degrees $\deg(f^n, q)$ are  uniformly bounded by Lemma~\ref{lem:cycle}; in particular,  there exists 
$N\in \N$ such that $\beta_f(p)\le N$ for all $p\in \CDach$. Moreover, it is clear that $\beta_f(p)=1$ for $p\in 
\CDach_P=\CDach\setminus \post(f)$ and $\beta_f(p)>1$ for $p\in \post(f)$. 
 If $p\in \post(f)$, then, as we will see, the above claim is true with the exponent $\alpha=2-2/\beta_f(p)\in (0,2)$.

   To prove the claim, we fix $w_0 \in \CDach _P$ and $p\in \post(f)$.
   We pick a Jordan curve $\CC\sub \CDach$ with $\post(f)\sub \CC$, and consider cells for $(f,\CC)$. 
   
   Let $V=W^0(p)$ be the $0$-flower of $p$. This is a simply connected region
   whose complement contains more than two points (this complement contains the $0$-vertices distinct from $p$, of which there are $\#\post(f)-1\ge 2$). So  there exists a conformal map 
   $\psi\:  V\ra \D $ such that $\psi(p)=0$.  
   Let $n\in \N_0$ and consider a component  $U$ of the preimage $f^{-n}(V)$. By 
   Lemma~\ref{lem:mapflowers}~\ref{item:mapflowers2} we know that $U$  is an $n$-flower and so there exists an $n$-vertex  $q$
   such that $U_q\coloneqq W^n(q)=U$. Then necessarily $f^n(q)=p$, and 
   we have 
   $$ \bigcup_{q\in f^{-n}(p)} U_q=f^{-n}(V). $$ 
For such a component $U=U_q$ there also exists a conformal map $\varphi\:  U_q\ra \D  $ 
with $\varphi(q)=0$.  Since $U_q$ is a component of $f^{-n}(V)$, the map $f^n|U_q\: U_q \ra V$ is proper (see Lemma~\ref{lem:proper}~\ref{item:proper2}). It follows that the map $h\coloneqq  \psi\circ (f^n|U_q)\circ \varphi^{-1}\: \D \ra \D$ is also proper, and 
hence a finite Blaschke product 
(see \cite[Exercise~6.12]{Bur0}). 
Now $(f^n|U_q)^{-1}(p)=\{q\}$ which implies 
that $h^{-1}(0)=\{0\}$. It follows that  $h(z)=cz^k$ with $c\in \C$, $|c|=1$, and $k\in \N$.
Since we can postcompose $\varphi$  with a rotation if necessary, we may assume that 
$c=1$. We then obtain  the following commutative diagram: 
   \begin{equation}\label{eq:commdist}
    \xymatrix@C+2pc{
      U_q \ar[d]^{\varphi}   \ar[r]^{f^n} & V \ar[d]_{\psi}
      \\
      \D \ar[r]^{ h(u)=u^{k}} & \D 
      \rlap{.}
    }
  \end{equation}
 Here we have the uniform bound $k\le \beta_f(p)\leq N$, where
 $\beta_f(p)<\infty$ is defined as in \eqref{eq:defbetaf}.
 
In the commutative diagram   \eqref{eq:commdist} the map $\psi$
is fixed for the given point $p$ under consideration. On the
other hand,  
the map $\varphi$ (and  also $k$) depend 
on $n$ and the chosen preimage $q\in f^{-n}(p)$. In the ensuing 
argument it will be crucial that we obtain uniform estimates for all such maps $\varphi$ (as well as for the points under consideration).

We now fix a compact subset $K$ of $V$, say
$K= \psi^{-1}(\overline B_\C(0,1/2))\sub V$, and use
$w_1\coloneqq \psi^{-1}(1/2)\in K$ as a base point in $K$.  Note that $K$ is a
compact neighborhood of $p$. We want to estimate $M(z,w_1)$ for
$z\in K\setminus\{p\}$.  For this we take preimages of $z\in K$
under maps as in the diagram \eqref{eq:commdist} and compare the
values of Jacobians at such preimages with corresponding values
at preimages of the basepoint $w_1$.  Here it will be important
that the preimage of $K$ under a map
$f^n\circ \varphi^{-1}=\psi^{-1} \circ h\colon \D\to V$ as in \eqref{eq:commdist} lies
in the fixed compact subset
\begin{equation*}
  A
  \coloneqq 
  \overline B_\C(0, 2^{-1/N})
  \supset   \overline B_\C(0, 2^{-1/k}) =
  h^{-1}(\overline{B}_{\C}(0, 1/2)) =
  h^{-1}(\psi(K))
 \end{equation*}
of $\D$.

  So let $z\in K$ be arbitrary and suppose $w',z'\in U_q$ are points such
  that $f^n(z')=z$ and $f^n(w')=w_1$. Let $u'=\varphi(z')$
  and $v'=\varphi(w')$.  Finally, define
  $u=h(u')=(u')^k=\psi(z)$ and $v=h(v')=(v')^k=\psi(w_1)$.  Then $u'$
  and $v'$ lie in the fixed compact subset $A$ of $\D$. The
  situation is represented by the following commutative diagram:
\begin{equation*}
  \xymatrix@C+2pc{
    z',w'\in U_q  \ar[d]^{\varphi}\ar[r]^{f^n} & z,w_1 \in K \subset V  \ar[d]_{\psi}
    \\
    u',v'\in A\subset\D \ar[r]^{h}  
    & u,v\in \overline{B}_{\C}(0, 1/2)\subset \D
      \rlap{.}
    }
\end{equation*}

It follows from Koebe's distortion theorem  (as formulated in \eqref{eq:Koebe_dg} of Theorem~\ref{thm:Koebe}) that  
  \begin{equation}\label{eq:comp} 
   J_{\varphi^{-1}}(u')=\big((\varphi^{-1})^\sharp(u')\big)^2 \asymp \big((\varphi^{-1})^\sharp(v')\big)^2 =
   J_{\varphi^{-1}}(v')
   \end{equation} with $C(\asymp)$ independent of $u'$, $v'$, and $\varphi$.
  Note that the flower $U_q=\varphi^{-1}(\D)$ is not necessarily contained in a hemisphere and so the assumption in Theorem~\ref{thm:Koebe} may not always be true;  it is true 
  with possibly  finitely many exceptions $n\in \N_0$ and $q\in f^{-n}(p)$. Then 
  \eqref{eq:comp} still holds with a uniform constant if we adjust  the constant to account for the finitely many exceptional cases (see the discussion after the proof of Theorem~\ref{thm:Koebe}). 
  
  Since $u,v\in \psi(K)=\overline B_\C(0,1/2)$
  and $\psi\: V \ra \D$ is a fixed conformal map (for the given point  $p$), we also have 
$$ J_{\psi^{-1}}(u)\asymp J_{\psi^{-1}}(v)
\asymp 1$$
  with $C(\asymp)$ independent of $u$ and $v$. 
In addition,   $$ |z-p| = |\psi^{-1}(u)-\psi^{-1}(0)|\asymp |u|$$
  with $C(\asymp)$ independent of $z$.  Finally, note that  $$J_h(u')\asymp |u'|^{2k-2}$$ and 
  $$ J_h(v')\asymp 1$$
 with $C(\asymp)$ independent of the choices, because $k$ is uniformly bounded by $\beta_f(p)$.
  
  Putting this all together, we arrive at
  \begin{align*} 
   \frac {J_{f^n}(w')}{J_{f^n}(z')}&= \frac
   {J_{\psi^{-1}}(v) J_h(v') J_{\varphi^{-1}}(v')^{-1}} 
   { J_{\psi^{-1}}(u) J_h(u')  J_{\varphi^{-1}}(u')^{-1}}\\
&  \asymp   \frac
   {J_h(v')} {J_h(u')}\asymp{|u'|^{-2k+2}}=|u|^{-2+2/k}\\
   &\le |u|^{-2+2/\beta_f(p)} \asymp |z-p|^{-2+2/\beta_f(p)}.
    \end{align*} 
  It follows that  
  \begin{align*} \rho_n(z)&=\sum_{z'\in f^{-n}(z)} J_{f^n}(z')^{-1}
  =\sum_{q\in f^{-n}(p)} \sum_{z'\in U_q\cap  f^{-n}(z)} J_{f^n}(z')^{-1}\\
  &\lesssim   |z-p|^{-2+2/\beta_f(p)} \sum_{q\in f^{-n}(p)} \sum_{w'\in U_q\cap  f^{-n}(w_1)} J_{f^n}(w')^{-1}\\ & = |z-p|^{-\alpha} \rho_n(w_1), 
   \end{align*} 
  where $\alpha =2-2/\beta_f(p)\in (0,2)$. Here we used the previous estimate and also the fact that $w_1$ and $z$ have the same number of  preimages in 
  each flower $U_q$, namely $k=\deg(f^n, q)$ preimages. 
  Since the implicit constants here are independent of 
  $z\in K$ and $n\in \N_0$,
  we conclude that 
  $$ M(z,w_1)=\sup_{n\in \N_0} \frac{\rho_n(z)}{\rho_n(w_1)} \lesssim  |z-p|^{-\alpha}$$ 
  for $z$ near $p$. 
 Hence 
 \begin{equation}
   \label{eq:Mest}
   M(z,w_0)\le M(z,w_1)M(w_1, w_0)\lesssim M(z,w_1)\lesssim  |z-p|^{-\alpha} 
 \end{equation}
 for $z$ near $p$. 
The claim and  statement \ref{item:Mzw0_bounded} follow.  \end{proof}

\begin{lemma} [Fixed point of ${\mathcal{R}}$]
  \label{lem:fixed-point-widet}
  There exists a positive function $\rho\in C(\CDach_P)$ with 
  $\int_{\CDach}\rho\, d\leb=1$ such that 
  ${\mathcal{R}}(\rho)=\rho$. 
  \end{lemma}
  A fixed point $\rho$ of ${\mathcal{R}}$ with these properties is actually unique. We will not show this directly, but it will immediately follow from the uniqueness of the measure $\lambda_f$ in Theorem~\ref{thm:ex_inv_abs_L}. 

\begin{proof}
We  define 
    \begin{equation*}
    \widetilde{\rho}_n =  \frac{1}{n}\sum_{i=0}^{n-1} \rho_i
    =
    \frac{1}{n}\sum_{i=0}^{n-1} 
   {\mathcal{R}}^i(1)
  \end{equation*}
  for $n\in \N$. 
  Then each $\widetilde{\rho}_n$ is a positive function in $C(\CDach_P)$. 
  
  \smallskip 
  {\em Claim 1.} The functions  $\widetilde{\rho}_n$ are  locally
  uniformly bounded on $\CDach_P$ for $n\in \N$.

  
  \smallskip 
 To prove this  claim,  it suffices to show that the functions  
 $\rho_n={\mathcal{R}}^n(1)$, $n\in \N_0$, are  
  locally uniformly bounded
  on $\CDach_P$. 
 To see this,  pick   a point  $w_0\in \CDach_P$ and a 
   small $\delta>0$  such that the disk  $B\coloneqq \overline B(w_0,\delta)$ is contained in 
 $\CDach_P$. 
 We will first produce an upper bound for $\rho_n(w_0)$.  
 
 Let $M$ be the function defined in \eqref{eq:defMzw}.
  Lemma~\ref{lem:defcontrolM}~\ref{item:defcontrolM0} implies that the function $z\mapsto M(w_0, z)$ is uniformly bounded from above for $z\in B$, say by $C_0\ge 1$. Then for
  each  $n\in \N_0$ we have 
  \begin{align*}
 {\rho}_n(w_0)&\le \frac{1}{\leb(B)}\int_B M(w_0,z){\rho_n}(z) \, d\leb(z)\\ 
    &\leq  \frac{C_0}{\leb(B)}   \int_{\CDach} {\rho}_n \, d\leb=  \frac{C_0}{\leb(B)}=:C_1.
  \end{align*}
 Here we used \eqref{eq:Mcontrolsrhon} and the normalization~\eqref{eq:rhonnormal}.  Hence 
  $$ \rho_n(z)\le M(z,w_0)\rho_n(w_0)\le C_1 M(z,w_0)$$
  for $n\in \N_0$ and $z\in \CDach_P$. 
   
   Since $M$ is bounded on compact subsets of 
  $\CDach_P\times \CDach_P$ as 
follows 
from   Lemma \ref{lem:defcontrolM}~\ref{item:defcontrolM0}, this last inequality 
   implies  that the functions  $\rho_n$, $n\in \N_0$, are locally uniformly bounded on $\CDach_P$. Claim 1 follows.

\smallskip
{\em Claim 2.} At each point $z_0\in \CDach_P$ the functions 
$\widetilde \rho_n$, $n\in \N_0$, are equicontinuous.

\smallskip 
First note that \eqref{eq:Mcontrolsrhon} implies that 
\begin{equation}\label{eq:Mcontrolsrhotid}
\widetilde \rho_n(z)\le M(z,w) \widetilde \rho_n(w)
\end{equation} 
for all $n\in \N$ and $z,w\in \CDach_P$. 

Let  $z_0\in \CDach_P$ be arbitrary. Then it follows from \eqref{eq:Mcontrolsrhotid}  that  for $n\in \N$ 
 and $z\in \CDach_P$ we have 
$$ \biggl| \frac {\widetilde \rho_n(z)}{\widetilde \rho_n(z_0)}-1\biggr| \le \max\{ M(z,z_0)-1, 
|M(z_0,z)^{-1}-1|\}.
$$ 
By Claim~1 we know that there exists a constant $C>0$ such that $  \widetilde \rho_n(z_0)\le C$ for $n\in \N$. It follows that for each $z\in \CDach_P$ and $n\in \N$ we have
\begin{align*}
|\widetilde \rho_n(z)-\widetilde \rho_n(z_0)|&= |\widetilde \rho_n(z_0)|\cdot  \biggl| \frac {\widetilde \rho_n(z)}{\widetilde \rho_n(z_0)}-1\biggr|\\
&\le C \max\{ M(z,z_0)-1, 
|M(z_0,z)^{-1}-1|\}.
\end{align*}
Since $M(z,z_0), M(z_0, z)\to M(z_0,z_0)=1$ as $z\to z_0$, Claim~2 follows. 

\smallskip 
  By what we have seen, the functions $ \widetilde \rho_n$, $n\in \N$, form a locally uniformly bounded family of continuous functions on $ \CDach_P$ that is equicontinuous at each point in 
  $ \CDach_P$.   
So the  Arzel\`a-Ascoli theorem  implies 
that  there exists a  subsequence
  $\{\widetilde{\rho}_{n_k}\} $ that converges locally uniformly
  on $\CDach _P$ to a non-negative function ${\rho}\in C(\CDach _P)$.   
  
  \smallskip
  {\em Claim 3.}    ${\mathcal{R}}({\rho})={\rho}$.
  
  \smallskip 
 To see this, let 
  $z\in \CDach_P$. As we have shown in the proof of Claim~1,  the functions $\rho_n$ are uniformly bounded at $z$,
   say by the constant $C>0$. So  for $n\in \N$ we have 
  \begin{align*}
   |{\mathcal{R}}(\widetilde{\rho}_n)(z)-\widetilde{\rho}_n(z)|
    &= 
    \frac1n |{\mathcal{R}}^{n}(1)(z)-
   {\mathcal{R}}^{0}(1)(z)|   \\ &=  \frac1n |\rho_n(z)-\rho_0(z) |\leq \frac{2C}{n} \to 0 
    \text{ as $n\to \infty$}.
  \end{align*}
  Hence the continuity  of $\mathcal{R}$ (see Lemma~\ref{lem:Rprops}~\ref{item:Ruelle1b}) implies that 
  \begin{equation*}
    {\mathcal{R}}({\rho})(z)
    = 
    \lim_{k\to \infty} {\mathcal{R}}(\widetilde{\rho}_{n_k})(z)
    =
    \lim_{k\to \infty} \widetilde{\rho}_{n_k}(z)={\rho}(z).
  \end{equation*}
 Claim~3 follows.
  \smallskip 

  The proof will be complete if we establish the last claim. 

 \smallskip 
{\em Claim 4.} The function $\rho$ is positive and satisfies $\int_{\CDach} \rho\, d \leb=1$.

\smallskip
By the normalization \eqref{eq:rhonnormal} and the definition of $\widetilde \rho_n$ we have 
$$ \int_{\CDach} \widetilde \rho_n\, d\leb =1 $$
for $n\in \N$.
Pick  a point $w_0\in \CDach_P.$ Then by Claim 1 there exists $C>0$ such that $\widetilde \rho_n(w_0)\le C$ for $n\in \N$, and so by  \eqref{eq:Mcontrolsrhotid} we have 
$$  0\le \widetilde \rho_n(z)\le CM(z,w_0)$$
for $n\in \N$ and $z\in \CDach_P$. 
Hence  by Lemma~\ref{lem:defcontrolM}~\ref{item:Mzw0_bounded} the
functions $ \widetilde \rho_n$, $n\in \N$, are majorized by an
integrable function
and so by Lebesgue's dominated convergence theorem we conclude that 
$$ \int_{\CDach}  \rho\, d\leb= \lim_{k\to \infty} \int_{\CDach} \widetilde \rho_{n_k}\, d\leb=1. $$

\smallskip
We know that $\rho$ is non-negative on $\CDach_P$. To see that  
$\rho$ is actually positive, we argue by contradiction and assume that there exists a point
$w\in \CDach_P$ such that $\rho(w)=0$. In inequality \eqref{eq:Mcontrolsrhotid} we can pass to  sublimits and conclude that then 
$$ 0\le \rho(z)\le M(z,w) \rho(w)=0$$ 
 and so $\rho(z)=0$ for all $z\in 
\CDach_P$. This is impossible, since 
$\int_{\CDach} \rho\, d \leb=1$. 
  \end{proof}

We are now ready to prove the existence and uniqueness of an $f$-invariant probability measure  that 
is absolutely continuous with respect to $\leb$.

\begin{proof}
  [Proof of Theorem~\ref{thm:ex_inv_abs_L}]
 Let $\rho$ be a fixed point of the Ruelle operator as provided
 by Lemma~\ref{lem:fixed-point-widet} and 
let $\lambda$ be the measure 
with $d\lambda=\rho\, d\leb$. 
Then $\lambda_f\coloneqq\lambda$ is a 
probability measure that is absolutely continuous with respect to 
 $\leb$. Actually, since $\rho$ is a positive continuous function 
 on $\CDach$ outside the finite set $\post(f)$, the measures $\lambda$ and $\leb$ are mutually absolutely continuous. 
 By Lemma~\ref{lem:Rprops}~\ref{item:Ruelle2} the measure $\lambda$ is $f$-invariant. The existence of a measure with the desired properties follows.

By Corollary~\ref{cor:ergod}
 each $f$-invariant measure that is absolutely continuous with respect to $\leb$ is 
  ergodic for  $f$. If $\mu$ is another $f$-invariant probability measure that is
  absolutely continuous with respect to $\leb$, then it is also absolutely continuous with respect to $\lambda$. On the other hand, both measures are ergodic and so 
necessarily $\mu=\lambda$. The uniqueness of $\lambda$ follows. 
\end{proof}

We now take a closer look at the unique measure $\lambda_f$  
for a rational expanding Thurston map $f$ with a
parabolic orbifold $\mathcal{O}_f=(\CDach, \alpha_f)$. Since $f$ has no periodic critical points, this is the case precisely when $f$ is a Latt\`es map 
(see Theorem~\ref{thm:Lattesstruc}~\ref{item:Lattessrucii}). 
\index{parabolic!orbifold} 
\index{orbifold!parabolic} 
\index{Thurston map!parabolic}
\index{Latt\`{e}s map}

By Theorem~\ref{thm:Lattesstruc}~\ref{item:Lattessruciii} we know that there exists a map  
 $A\: \C \ra \C$ of the form $A(u)=\alpha u+\beta$ with $\alpha,\beta\in \C$, $\alpha\ne 0$, such that $f\circ \Theta= \Theta\circ A$, 
 where $\Theta\: \C \ra \CDach$ is the universal orbifold covering map of $\mathcal{O}_f$. Moreover, 
    $d\coloneqq \deg(f)=|\alpha|^2$ (see Lemma~\ref{lem:deglatttype}). 
    
 We first want to construct a measure $\Om$ on $\CDach$ so that $f$ has constant Jacobian with respect to $\Om$. Since the lift $A$ of $f$ by $\Theta$ has constant Jacobian $J_{A,\leb_{\C}}=|\alpha|^2=d$
 with respect to Lebesgue measure $\leb_{\C}$ on $\C$, we want to 
(locally) 
push forward\index{push-forward!of measure!by orbifold covering map}\index{canonical orbifold!measure}\index{measure!canonical orbifold}\index{orbifold!canonical measure}\index{OAA@$\Omega,\Omega_f$}\index{push-forward!of measure!by orbifold covering map} 
 $\leb_\C$ by $\Theta$ to define $\Om$. The measure obtained in this way is actually the canonical orbifold measure $\Om=\Om_f$
of $\mathcal{O}_f$  
  
We will  quickly review the  definition of  $\Om$, but   refer to 
Section~\ref{sec:expratThmaps} for more details.  Let $J_\Theta$ be the Jacobian of $\Theta$ with 
$ \leb_{\C}$ being the underlying measure on the source $\C$ and $\leb_{\CDach}$
being the measure on the target space $\CDach$ of $\Theta$.
Then 
$$ J_\Theta(u)= \frac{|\Theta'(u)|^2}{\pi(1+|\Theta(u)|^2)^2}$$ 
for $u\in \C$. 
Let  $\kappa\: \CDach_P \ra (0, \infty)$ be the function defined as 
$$\kappa (w)= J_\Theta(u)^{-1}$$ 
for $w\in \CDach_P$, where $u\in \Theta^{-1}(w)$. One can show that this function is well-defined,
positive and continuous on $\CDach_P$, and integrable with respect to 
$\leb_{\CDach}$. The map $\Theta$ is only unique up to a precomposition with a conformal automorphism of $\C$. Choosing this automorphism appropriately, we may assume 
that $\int_{\CDach}\kappa\, d \leb_{\CDach}=1.$
Now let  $\Om$ be the unique measure on 
$\CDach$ with Radon-Nikodym derivative $\kappa$ with respect to 
$\leb_{\CDach}$, meaning that
$d\Omega = \kappa \,d\leb_{\CDach}$. 
This measure  is  normalized so that $\Om(\CDach)=1$.

 By using the identity $f\circ \Theta= \Theta \circ A$ and the chain rule for Jacobians (where  $\leb_{\CDach}$ is  the measure on $\CDach$ and  $\leb_\C$ the measure on $\C$) we see that
 \begin{equation}\label{eq:Jacobform}
  J_f(\Theta(u))\cdot  J_\Theta(u)= |\alpha|^2\cdot J_\Theta (A(u))=d \cdot J_\Theta (A(u))
  \end{equation} 
 for $u\in \C$.
 Here $J_f=(f^\sharp)^2$ (see \eqref{eq:formula_J2}). 

Now suppose $w\in \CDach_P$ and $z\in f^{-1}(w)$. If we pick a point 
$u\in \C$ such that $\Theta(u)=z$, then for $v\coloneqq A(u)$ we have 
$$ \Theta(v)= (\Theta\circ A)(u)=(f\circ \Theta)(u)=f(z)=w. $$
It follows that 
$$ \kappa (z) = J_\Theta(u)^{-1} \text{ and }  \kappa(w)= J_\Theta(v)^{-1}, $$
and so by taking reciprocals in  \eqref{eq:Jacobform} we obtain 
\begin{equation}\label{eq:Jacobform2}
 \kappa (z)\cdot J_f(z)^{-1}= \frac 1d  \kappa (w)
\end{equation}
 whenever $w\in \CDach_P$ and $z\in f^{-1}(w)$.

From this and  \eqref{eq:formula_J} we conclude that if we  change the underlying measure on $\CDach$ from $\leb_{\CDach}$ 
to $\Om=\Om_f$, then the Jacobian of $f$  is 
given by 
\begin{equation}\label{eq:constJac} 
 J_{f,\Om}(z)= \frac{\kappa(f(z))}{\kappa(z)}J_f(z)=d=\deg(f)
 \end{equation} 
for $z\in \CDach\setminus f^{-1}(\post(f))$. So $J_{f,\Om}(z)=d$ for  $\Om$-almost every
$z\in \CDach$ and $J_{f,\Om}$ is indeed constant. 

Since each point $w\in \CDach_P$  has precisely $d$ preimages under $f$, it also  follows from \eqref{eq:Jacobform2} that
$$ \mathcal{R}(\kappa)(w) =\sum_{z\in f^{-1}(w)} \kappa (z) J_f(z)^{-1} =\kappa(w).
$$ 
This shows that $\kappa$ is a fixed point of the Ruelle operator with   properties as in 
Lemma~\ref{lem:fixed-point-widet}. Since $\kappa=d\Om/d\leb$ it follows from the uniqueness part of 
Theorem~\ref{thm:ex_inv_abs_L} that $\Om=\lambda_f$ whenever $f$ is a Latt\`es map.

The proof of Theorem~\ref{thm:Lattes_can_orb_meas} is now easy.

\begin{proof}[Proof of Theorem~\ref{thm:Lattes_can_orb_meas}] 
  Let $f\colon \CDach \to \CDach$ be a Latt\`{e}s map. As we have seen in the previous discussion, then  $\lambda_f=\Om_f$, where $\lambda_f$ is the $f$-invariant measure provided by 
  Theorem~\ref{thm:ex_inv_abs_L} and $\Om=\Om_f$ is the normalized canonical orbifold measure of $\mathcal{O}_f$.
   
  By Rokhlin's formula \eqref{eq:Rohlinsform} and \eqref{eq:constJac} the measure-theoretic entropy $h_\Om(f)$ of $\Om$ 
  is given by 
  $$  h_\Om(f)= \int_{\CDach} \log (J_{f,\Om})\, d\Om = \log(\deg(f))=h_{top}(f).$$
  So $\Om$ is a measure of maximal entropy. Since the measure of maximal entropy $\nu_f$ is uniquely determined for the  expanding Thurston map $f$
  (see Theorem~\ref{thm:maxentr0}), it follows that 
  $\nu_f=\Om_f=\lambda_f$ as desired. 
     \end{proof}

\section
[Latt\`{e}s maps, entropy, and Lebesgue measure]
{\for{toc}{Latt\`{e}s maps, entropy, and Lebesgue measure}\except{toc}{Latt\`{e}s maps, the measure of maximal entropy, and Lebesgue measure}}

\label{sec:lyapunov-exponent-mu}

 In this section we prove Theorem~\ref{thm:abscontimpLattes},
the special case of Zdunik's theorem. Let
$f\colon \CDach\to \CDach$ be a rational expanding Thurston map
and $\mu$ be an $f$-invariant probability measure on $\CDach$.
We know from Chapter~\ref{cha:measure} that the
(measure-theoretic) entropy $h_\mu(f)$ of $\mu$ satisfies
$h_\mu(f) \leq \log(\deg(f))$ (see Corollary~\ref{cor:topent} and
\eqref{eq:var_princ}). Moreover, here we have equality  precisely
for $\mu=\nu_f$, the measure of maximal entropy of $f$.

We first establish a characterization of Latt\`es maps that uses the measure  
provided by Theorem~\ref{thm:ex_inv_abs_L}.

\begin{theorem}
  \label{thm:Zdunik}
  \index{Latt\`{e}s map}
  \index{laa@$\lambda_f$}
  \index{entropy!measure-theoretic}
  \index{h mu@$h_\mu$}
  \index{measure!of maximal entropy}
  \index{measure-theoretic entropy}
  \index{n@$\nu_f$}
  Let $f\colon \CDach \to \CDach$ be a rational expanding
  Thurston map, and $\lambda=\lambda_f$ be the unique $f$-invariant probability
  measure that is absolutely continuous with respect to $\leb_{\CDach}$. Then for
  the entropy $h_\lambda(f)$ we have 
  \begin{equation*}
    h_\lambda(f) = \log (\deg(f)) 
    \text{ if and only if $f$ is a Latt\`{e}s map.} 
  \end{equation*}
\end{theorem}
In other words, $\lambda_f$ is equal to the measure of maximal entropy $\nu_f$ if and only if $f$ is a Latt\`es map. 

We require  the following lemma.
%

\begin{lemma}
  \label{lem:parabolic_characterization}
  \index{orbifold!parabolic} 
  \index{Thurston map!parabolic}
  \index{parabolic!orbifold}
  Let $f\colon S^2 \to S^2$ be a Thurston
  map. Then $f$ has a parabolic orbifold  if and only  if  
  \begin{equation}
    \label{eq:Lattes_characterization}
    \deg(f^n,q)= \deg(f^n,q'),
  \end{equation}
 whenever $p\in
  \post(f)$, $n\in \N$, and $q,q'\in f^{-n}(p)\setminus \post(f)$.
\end{lemma}

\begin{proof} If $f$ has a parabolic orbifold, then 
it follows from condition \ref{item:Of_para3} in 
  Proposition~\ref{prop:parabolicOf} by induction
  that 
  $$ \alpha_f(q)\deg(f^n,q)=\alpha_f (p),$$
  whenever $p\in S^2$, $n\in \N$, and $q\in f^{-n}(p)$. 
  If in addition $q\notin \post(f)$, then $ \alpha_f(q)=1$  and so 
   \begin{equation*}
   \deg(f^n, q)= \alpha_f(p). 
  \end{equation*}
  Relation  \eqref{eq:Lattes_characterization} follows. 

  To show the other implication, 
  assume  that \eqref{eq:Lattes_characterization} holds. 
 In order to show that $f$ has parabolic orbifold, we want to verify 
 condition \ref{item:Of_para3} in  Proposition~\ref{prop:parabolicOf}.
 We first establish several claims.

\smallskip 
  \emph{Claim 1.} Let  $p\in S^2$. Then  we have 
$    \deg(f^n, q)= \deg(f^m,q'), $ whenever $n,m\in \N$,  
  $q\in f^{-n}(p)\setminus \post(f)$, and $q'\in
  f^{-m}(p)\setminus \post(f)$.  
 
  \smallskip 

  Indeed, suppose $q$ and $q'$ are as in this statement  for a given point $p\in S^2$. If
  $p\in S^2\setminus \post(f)$, then
  $\deg(f^n, q)=1= \deg(f^m,q')$ and the claim follows. So we may
  assume $p\in \post(f)$.
 
 We then choose $k,l\in \N$ such that
  $n+k= m+l$, and  pick points $u\in f^{-k}(q)$ and $u'\in f^{-l}(q')$. Then 
  $\deg(f^k,u)=1$ (otherwise $q\in \post(f)$), and so
  \begin{equation*}
    \deg(f^{n+k},u)= \deg(f^n, q)\deg(f^k,u)= \deg(f^n,q).
  \end{equation*}
  Similarly,  $\deg(f^{m+l},u')= \deg(f^m,q')$. 
  
  Now $p\in \post(f)$ and $u,u'\in f^{-(n+k)}(p)\setminus \post(f)$. So
  \eqref{eq:Lattes_characterization} implies that 
  \begin{equation*}
    \deg(f^n,q)= \deg(f^{n+k}, u)= \deg(f^{m+l}, u')= \deg(f^m,q').
  \end{equation*}
  Claim~1 follows.  

\smallskip
{\em Claim 2.} If $p\in S^2$ and $\bigcup_{i\in \N}f^{-i}(p)\sub \post(f)$, then $p$ is contained in a critical cycle  of $f$. 

\smallskip
If $p\in S^2$ is as in this statement, then each 
preimage of $p$ under $f$ is in $\post(f)$. Since $\post(f)$ is
$f$-invariant, it follows that $p\in \post(f)$.  
Hence there exist $q\in \crit(f)$ and $n\in \N$ 
such that $f^n(q)=p$. To prove the claim, it suffices to show that $q$ is periodic, because then $p$ is contained in the critical cycle generated by $q$. 

To see that $q$ is periodic, let $q_1\coloneqq q$ and inductively choose points $q_k\in S^2$ such that $f(q_{k+1})=q_{k}$ for $k\in \N$. Each point  $q_k$ is a preimage of $p$ under some iterate of $f$. Hence $q_k\in \post(f)$.
Since  $\post(f)$ is a finite set, not all the points $q_k$, $k\in \N$,  can be distinct.
So there exist $k,l\in \N$ such that $q_{k+l}=q_k=f^l(q_{k+l})$. We see that  
$q_{k+l}$ is a periodic point of $f$. This implies that $q=q_1=f^{k+l-1}(q_{k+l})$ is also a periodic point of $f$ and Claim~2 follows.

  \smallskip
  \emph{Claim 3.} The ramification function of $f$ satisfies 
  \begin{equation}\label{eq:cl2Latt}
    \alpha_f(p) = \deg(f^n,q),
  \end{equation}
  whenever 
   $p\in S^2$,  $n\in \N$, and  $q\in f^{-n}(p)\setminus \post(f)$. 
  \smallskip 

To see this, let  $p\in S^2$ be arbitrary. We may assume that 
the set $\bigcup_{i\in \N} f^{-i}(p)$ is not contained in $\post(f)$, because otherwise there is nothing to prove. 
 
  We know that  $\alpha_f(p)$  is  the
  least  common multiple of all numbers $\deg(f^k,u)$, where
  $k\in \N$ and $u\in f^{-k}(p)$ (see Definition~\ref{def:weightf}). 
  Moreover, by Claim~1 the right hand side in \eqref{eq:cl2Latt} is independent of 
 the choices of   $n$ and $q$. So in order to prove \eqref{eq:cl2Latt}, it suffices 
 to show that if  $k\in \N$ and  $u\in f^{-k}(p)$, 
then there exist $n\in \N$ and   
$q\in f^{-n}(p)\setminus \post(f)$ such that  $\deg(f^{n},q)$ is a multiple of $\deg(f^k,u)$.

If $\bigcup_{i\in \N} f^{-i}(u)\sub \post(f)$, then by Claim~2 the point $u$ belongs to a critical cycle of $f$. This cycle then contains also 
$p$. Therefore, $p$ is a preimage of $u$ under some
iterate of $f$. This in turn gives 
 $$\bigcup_{i\in \N} f^{-i}(p)\sub 
\bigcup_{i\in \N} f^{-i}(u)\sub \post(f), $$  which is a contradiction to our 
additional assumption on  $p$. 

So $\bigcup_{i\in \N} f^{-i}(u)$ is not contained in $\post(f)$. 
Then  we can find  $l\in \N$ and  a point $q\in f^{-l}(u)$ that is not a postcritical point of $f$. Setting $n=k+l$,
  we have $q\in f^{-n}(p)\setminus \post(f)$. Moreover,    
   $\deg(f^{n},q)=
  \deg(f^k,u)\deg(f^l,q)$ is a multiple of $\deg(f^k,u)$. 
  Claim~3  follows.

\smallskip
 After these preparations we will now show that condition \ref{item:Of_para3} in  Proposition~\ref{prop:parabolicOf} is true. So  let $p,q\in \CDach$ with
  $f(q)=p$ be arbitrary. If $\bigcup_{n\in \N} f^{-n}(q)\sub \post(f)$,
  then by Claim~2 the point $q$, and hence also $p=f(q)$, belongs to a critical cycle of $f$. Then $\alpha_f(p)=\infty=\alpha_f(q)= \alpha_f(q) \deg(f,q)$ by Proposition~\ref{prop:otherramprops}~\ref{item:rami_infty}. 
  
If  $\bigcup_{n\in \N} f^{-n}(q)$ is not contained in $\post(f)$, then we 
 can find  $n\in \N$ and a point  $u\in
  f^{-n}(q)$ that is not a postcritical point. Then $f^n(u)= q$
  and $f^{n+1}(u)=p$. Thus $\alpha_f(q)= \deg(f^n,u)$ and
  $\alpha_f(p)= \deg(f^{n+1},u)$ by Claim 3. Therefore,
  \begin{equation*}
    \alpha_f(p)= \deg(f^{n+1}, u)= \deg(f^n, u)\deg(f, q)=
    \alpha_f(q)\deg(f,q). 
  \end{equation*}
  We see that  condition \ref{item:Of_para3} in   Proposition~\ref{prop:parabolicOf} is indeed  satisfied. This shows that $f$ has a parabolic orbifold. 
\end{proof}

\begin{proof}[Proof of Theorem~\ref{thm:Zdunik}]
  Let $f\colon \CDach \to \CDach$ be a rational expanding
  Thurston map, $\lambda=\lambda_f$ be the unique  measure
   on $\CDach$ given by
  Theorem~\ref{thm:ex_inv_abs_L}, and $h_\lambda(f)$ be the entropy of $\lambda$.   
  
  By Rokhlin's formula \eqref{eq:Rohlinsform} we have
  \begin{equation*}
    h_\lambda(f) = \int_{\CDach} \log (J_{f,\lambda}) \, d\lambda.
  \end{equation*}
  Combined with  \eqref{eq:Jdmu} and Jensen's inequality this gives 
  \begin{equation*}
    h_\lambda(f)  \leq \log\biggl( \int_{\CDach} J_{f,\lambda} \, d\lambda \biggr)= \log(\deg(f)),
  \end{equation*}
  where equality is achieved if and only if $J_\lambda= \deg(f)$ $\lambda$-almost everywhere and hence $\leb$-almost everywhere on $\CDach$.

If we  assume that  $f$ is a Latt\`{e}s map, then by 
Theorem~\ref{thm:Lattes_can_orb_meas} and by 
  \eqref{eq:constJac} we have $J_{f,\lambda}=\deg(f)$
  $\lambda$-almost everywhere. Thus  $h_\lambda(f)=\log(\deg(f))$. 
  
Conversely,  assume   that $h_\lambda(f) =\log(\deg(f))$ and so  $J_{f,\lambda}=
  h\coloneqq \log(\deg(f))$ $\lambda$-almost everywhere  on $\CDach$. Let  $\rho=d\lambda/d\leb$ be the Radon-Nikodym derivative of $\lambda$ with respect to Lebesgue measure $\leb$.  We know by Theorem~\ref{thm:ex_inv_abs_L} that this is a positive continuous function on $\CDach\setminus \post(f)$. Since the Jacobian is given by (see \eqref{eq:formula_J})
  \begin{equation*}
    J_{f,\lambda}(z)= f^\sharp(z)^2\frac {\rho( f(z))} {\rho(z)}= h,  
  \end{equation*}
  we conclude that 
  $$ 
\rho( f(z))=h \frac{ \rho (z)} {f^\sharp(z)^{2}} $$
  for $z\in \CDach \setminus f^{-1}(\post(f))$.

  If we iterate this relation and use the chain rule for the spherical derivative, we arrive at 
  \begin{equation}\label{eq:rhoiter}
  \rho( f^n(z))=h^n  \frac{ \rho (z)} {(f^n)^\sharp(z)^{2}} 
  \end{equation}
for $z\in \CDach \setminus f^{-n}(\post(f))$.

We will use this relation to derive the asymptotic behavior of $\rho$ near a point 
$p\in \post(f)$ in order to   verify the condition in Lemma~\ref{lem:parabolic_characterization}.
For this let $n\in \N$, and consider an arbitrary  point $q\in f^{-n}(p)\setminus \post(f)$.


 Let  $k=\deg(f^n, q)$. If   $z\in \CDach$ is a point near $q$,  then $w=f^n(z)$ is a point near $p$. By considering local power series expansions of $f^n$ in holomorphic coordinates, we see that 
 $$ |w-p|\asymp |z-q|^k$$ 
 and 
 $$ (f^n)^\sharp(z) \asymp |z-q|^{k-1} \asymp |w-p|^{1-1/k}, $$ 
  where  the constants $C(\asymp)$ are independent of $z$ near $q$
  (recall that  we use ``Polish notation'' $|u-v|$ for  the chordal distance between 
points $u,v\in \CDach$). 
 
Now we know  from the second part of Theorem~\ref{thm:ex_inv_abs_L} that $\rho(z) \asymp 1$ for  $z$ 
  near  $q$, because $q\not\in \post(f)$.  So from  relation \eqref{eq:rhoiter} we conclude that 
  $$\rho(w) \asymp \frac{\rho(z)}{(f^n)^\sharp(z)^2} \asymp |w-p|^{-2+2/k} $$
  for all $z$ near $q$, and hence for all $w$ near $p$.


 If $q'$ is another point with $q'\in f^{-n}(p)\setminus \post(f)$ and $k'\coloneqq \deg(f^n, q')$, then the same argument shows that
   \begin{equation*}
    \rho(w)\asymp \abs{w-p}^{-2+2/k'}
  \end{equation*}
  for all $w$ near $p$.
  These two estimates can only be valid if  $k=k'$. 
  
  We conclude that $\deg(f^n, q)=\deg(f^n, q')$ whenever $p\in \post(f)$, $n\in \N$, and $q,q'\in f^{-n}(p)\setminus \post(f)$.    So Lemma~\ref{lem:parabolic_characterization} implies that $f$ has a parabolic orbifold. 
Since $f$ is a rational expanding Thurston map, it has no periodic critical points (see  Proposition~\ref{prop:rationalexpch}). Hence  $f$ is a Latt\`{e}s map by 
  Theorem~\ref{thm:Lattesstruc}~\ref{item:Lattessrucii}.  
 \end{proof}
 
 Before we proceed to the proof of Theorem~\ref{thm:abscontimpLattes}, we record 
 a statement about the asymptotics of the Radon-Nikodym derivative $\rho=d \lambda_f/d\leb_{\CDach}$ near its  singularities that easily follows from our previous considerations.
For the formulation  we use the function 
$\beta_f\:\CDach\ra \N$ given by 
$$ \beta_f(p) =\max\{ \deg(f^n,q): n\in \N \text{ and } f^n(q)=p\} $$ 
for $p\in \CDach$ (see  
\eqref{eq:defbetaf}).

 \begin{prop} \label{rem:asymbevrho} Let $f\: \CDach\ra \CDach$
   be a rational expanding Thurston map and $\rho=d
   \lambda_f/d\leb_{\CDach}$ 
be the Radon-Nikodym 
derivative  
 of the measure given by Theorem~\ref{thm:ex_inv_abs_L}.   Then 
\begin{equation} \label{eq:rhoasymp}
\rho(w)\asymp |w-p|^{ -2+2/\beta_f(p)} 
\end{equation}
 for $w$ near $p\in \post(f)$. 
 \end{prop} 
 
 As we pointed out in the proof of Lemma~\ref{lem:defcontrolM}~\ref{item:Mzw0_bounded},
for a rational expanding Thurston map $f$ the function $\beta_f$ is bounded on $\CDach$ and we have $\beta_f(p)=1 $ for $p\in \CDach\setminus \post(f)$ and 
$\beta_f(p)\ge 2$  for $p\in \post(f)$.  Accordingly, the asymptotics \eqref{eq:rhoasymp}
also makes sense   near points in $\CDach \setminus \post(f)$, where  it should be interpreted as $\rho(w)\asymp 1$ for $w$ near $p\in \CDach \setminus \post(f)$. This is in accordance with the fact (see Theorem~\ref{thm:ex_inv_abs_L}) that $\rho$ is a positive continuous function on $\CDach \setminus \post(f)$. 

 \begin{proof} 
   Since $\rho$ is a fixed point of the Ruelle operator (see the proof of Theorem~\ref{thm:ex_inv_abs_L}), we 
 have 
 $$ \rho(w)=\sum_{z\in f^{-n}(w)} \rho(z)J_{f^n}(z)^{-1}$$
 for $w\in \CDach\setminus \post(f)$. 
 
 Fix $p\in\post(f)$. We can 
find  $n\in \N$ and $q\in f^{-n}(p)$ such that
$\deg(f^n,q) = \beta_f(p)$. Then clearly $q\notin \post(f)$, and so 
$\rho(z)\asymp 1$ for $z$ near $q$.

A point $w$ near $p$ has at least  one preimage  $z$ near
$q$ under $f^n$. As in the previous proof one sees that 
$$   J_{f^n}(z)=[(f^n)^\sharp(z)]^2 
  \asymp
  \abs{z-q}^{2(\beta_f(p)-1)}
  \asymp
  \abs{w-p}^{2-2/\beta_f(p)}.
$$ 
Thus we obtain the lower bound
\begin{equation*}
  \rho(w) \ge  \rho(z)  J_{f^n}(z)^{-1}\gtrsim \abs{w-p}^{-2+2/\beta_f(p)}.
\end{equation*}

For an inequality in the other direction, 
we note that by \eqref{eq:Mest} in the proof of 
 Lemma~\ref{lem:defcontrolM}~\ref{item:Mzw0_bounded}, 
we have 
$$ M(w,w_0)\lesssim  |w-p|^{ -2+2/\beta_f(p)}$$ for $w$ near $p$, where 
 $M$ is the function defined in \eqref{eq:defMzw} and $w_0\in \CDach\setminus \post (f)$ is a base point. Hence 
 $$ \rho(w) \lesssim   |w-p|^{ -2+2/\beta_f(p)} $$ as follows from the proof of 
 Lemma~\ref{lem:fixed-point-widet} (we have to pass to the sublimit $\rho$ in \eqref{eq:Mcontrolsrhotid}). The claim follows. 
 \end{proof} 

 Theorem~\ref{thm:abscontimpLattes}
is an easy consequence of Theorem~\ref{thm:Zdunik}.  

\begin{proof}[Proof of Theorem~\ref{thm:abscontimpLattes}]
  Let $f$ be a rational expanding Thurston map, $\nu_f$ be its
  measure of maximal entropy, and $\lambda_f$ be the unique $f$-invariant
 probability  measure that is absolutely continuous with respect to Lebesgue
  measure. 

 If $f$ is a Latt\`es map, then $\nu_f=\lambda_f$ by
  Theorem~\ref{thm:Lattes_can_orb_meas} and so $\nu_f$ is absolutely continuous with respect to $\leb$. 
  
 Conversely, suppose that  $\nu_f$ is absolutely continuous with respect to $\leb$. Since $\lambda_f$ and $\leb$ lie in the same measure class, $\nu_f$ is then also  absolutely continuous with respect to $\lambda_f$.  Now  $\nu_f$ and $\lambda_f$ are both ergodic $f$-invariant
  probability measures on $\CDach$. 
This implies that $\lambda_f=\nu_f$. Theorem~\ref{thm:Zdunik} 
then shows  
that $f$ is  a Latt\`es map.
  \end{proof}

We conclude this chapter  with  the proof of
Theorem~\ref{thm:S2vsf}~\ref{item:S2Lattes}.   
First we record  an
elementary lemma.

\begin{lemma}
  \label{lem:snow_Ahlfors}
  Let $\varphi\colon X\to X'$ be a snowflake homeomorphism
  between metric spaces $(X,d)$ and $(X',d')$ such that 
  \begin{equation*}
    d'(\varphi(x),\varphi(y)) \asymp d(x,y)^{\beta}
  \end{equation*}
  for all $x,y\in X$, where $\beta>0$ and $C(\asymp)$ are constants independent 
  of $x$ and $y$. Suppose that $Q>0$ and define $Q'= Q/\beta$.
  
 Then for the  corresponding Hausdorff measures we have 
\begin{equation}\label{eq:Hcompstr}
 \mathcal{H}^{Q'}_{d'}(\varphi(M))\asymp \mathcal{H}^Q_{d}(M)
 \end{equation} 
for each Borel set $M\sub X$ with $C(\asymp)$ independent of $M$.  Moreover, $(X,d)$ is Ahlfors $Q$-regular  if and only if $(X',d')$ is Ahlfors $Q'$-regular. \end{lemma}

\begin{proof} Relation~\eqref{eq:Hcompstr} follows from  
straightforward covering arguments; we skip the details. 

Suppose that $(X',d')$ is Ahlfors $Q'$-regular, and let $B$ be a closed ball in $X$ of radius 
$R\le \diam_d(X)$. Then there exist closed balls $B'$ and $B''$ in $X'$ with $B'\sub \varphi(B)\sub B''$ whose radii 
are comparable to $R'\coloneqq R^\beta\lesssim \diam_{d'}(X')$.  Then the Ahlfors $Q'$-regularity of $(X',d')$ implies that 
$$ \mathcal{H}^{Q'}_{d'}(\varphi(B))\asymp (R')^{Q'}\asymp R^Q,$$ and so
$$ \mathcal{H}^Q_d(B) \asymp  \mathcal{H}^{Q'}_{d'}(\varphi(B))\asymp R^Q.$$ This shows that 
$(X,d)$ is 
 Ahlfors $Q$-regular. 

The other implication is obtained by reversing the roles of $X$ and $X'$.
\end{proof}

\begin{proof}
  [Proof of Theorem~\ref{thm:S2vsf}~\ref{item:S2Lattes}] Let
  $\varrho$ be a visual metric for the expanding Thurston map
  $f\colon S^2\to S^2$. 

  Assume first that $f$ is topologically conjugate to a
  Latt\`{e}s map. Since such a topological conjugacy is in fact
  a snowflake homeomorphism with respect to visual metrics (see
  Proposition~\ref{prop:conjisom}), we can assume that
  $f\colon \CDach \to \CDach$ is a Latt\`{e}s map. Let $\sigma$
  be the chordal metric on $\CDach$, and $\omega$ be the
  canonical orbifold metric for $f$ (see Section~\ref{sec:expratThmaps} for the definition of $\omega$).\index{canonical orbifold!metric}\index{metric!canonical orbifold}\index{o@$\omega$}\index{orbifold!canonical metric}

  We know from  Proposition~\ref{prop:orbivispara} that $\omega$ is a visual metric for $f$. 
Since two visual metrics are snowflake equivalent
  according to
  Proposition~\ref{prop:visualsummary}~\ref{item:visd1d2}, we
 are  further reduced to the case  $\varrho=\omega$. Now for a Latt\`{e}s map the spaces 
  $(\CDach,
  \omega)$ and $(\CDach, \sigma)$  are bi-Lipschitz
  equivalent, because the orbifold of $f$ has no punctures (see
  Lemma~\ref{lem:om_chordal}~\ref{item:om_chordal_biLip}).
   It follows
  that $(S^2,\varrho)= (\CDach, \omega)$ is snowflake equivalent to $(\CDach,
  \sigma)$ as desired. 

   To prove the other implication, assume  that
  $(S^2, \varrho)$ is snowflake equivalent to $(\CDach,
  \sigma)$. In particular, these spaces are quasisymmetrically equivalent. Thus, by
  Theorem~\ref{thm:S2vsf}~\ref{item:S2qsphere}, the map $f$ is
  topologically conjugate to a rational map. So as before, we may 
  assume that $f\colon \CDach\to \CDach$ is in fact a rational
 expanding Thurston  map. 
Then $f$  does not have periodic critical points (Proposition~\ref{prop:rationalexpch}). In order to prove 
 that $f$ is a Latt\`es map, we will verify the condition in
   Theorem~\ref{thm:abscontimpLattes} and show that the measure
   of maximal entropy 
   $\nu_f$ of $f$ is absolutely continuous with respect to Lebesgue
   measure $\leb$ on $\CDach$.

 By our hypotheses,
there exists  a snowflake homeomorphism $\varphi \colon (\CDach,\varrho)\to
  (\CDach,\sigma)$. Then  
    \begin{equation*}
    \sigma(\varphi(x),\varphi(y))\asymp \varrho(x,y)^\beta
  \end{equation*}
  for all $x,y\in \CDach$, 
  where $\beta>0$ and  $C(\asymp)$ are 
   constants independent of $x$ and $y$. 
   
  Let $\mathcal{H}^2_\sigma$ denote $2$-dimensional Hausdorff measure on $\CDach$ with respect to the chordal metric $\sigma$, and $\mathcal{H}^{2\beta}_\varrho$
  denote $(2\beta)$-dimensional Hausdorff measure on $\CDach$ with respect to the  metric $\varrho$.
  
  Since $(\CDach, \sigma, \mathcal{H}^2_\sigma)$ is Ahlfors $2$-regular, 
  $(\CDach,  \varrho,  \mathcal{H}^{2\beta}_\varrho)$ is Ahlfors $(2\beta)$-regular 
  by Lemma~\ref{lem:snow_Ahlfors}.
   Moreover, 
   $$ \mathcal{H}^2_\sigma(\varphi(M))\asymp \mathcal{H}^{2\beta}_\varrho(M)$$
   for each Borel set $M\sub \CDach$.
   Since $\varrho$ is a visual metric and  $\mathcal{H}^{2\beta}_\varrho$ is Ahlfors regular, the measures $\nu_f$ and  $\mathcal{H}^{2\beta}_\varrho$ are comparable 
  (see the discussion after Proposition~\ref{prop:Ahlforsreg}). 
    
   Now  let $M\sub \CDach$ be an arbitrary  Borel set with $\leb(M)=0$, and define 
   $N=\varphi(M)$.  By Lemma~\ref{cor:visualqsmetric}  the map
  $\id_{\CDach}\colon (\CDach, \sigma) \to (\CDach,\varrho)$ is a
  quasisymmetry. This  implies that  the composition
  \begin{equation*}
    (\CDach,\sigma)
    \xrightarrow{\id} 
    (\CDach,\varrho) 
    \xrightarrow{\varphi} 
    (\CDach,\sigma),
  \end{equation*}
  i.e., the map $\varphi \: (\CDach, \sigma) \to (\CDach, \sigma)$,
  is also a quasisymmetry. Since quasisymmetries on $\CDach$ preserves sets of 
  (Lebesgue-) measure zero (see \cite[Theorem~3.4.1 and Theorem~3.1.2]{AIM}), we have  $\leb(N)=0$.  It follows that 
  $$ \nu_f(M)\asymp  \mathcal{H}^{2\beta}_\varrho(M)\asymp   \mathcal{H}^{2}_\sigma(\varphi(M))\asymp \leb(\varphi(M))=\leb(N)=0. $$
So $\nu_f$ is indeed absolutely continuous with respect to Lebesgue measure, and $f$ is a Latt\`es map by    Theorem~\ref{thm:abscontimpLattes}.    
\end{proof}

\ifthenelse{\boolean{singlechapter}}{ 
 
%


\chapter{A combinatorial characterization of Latt\`{e}s  maps}
\label{cha:latt-maps-comb}

In this chapter we characterize Latt\`es maps 
in terms of 
their combinatorial expansion behavior. This is based on results
by Qian Yin (see \cite{Qian15}). 

Let $f\: S^2\ra S^2$ be an expanding Thurston map, and $\CC\sub
S^2$ be a Jordan curve with $\post(f)\sub \CC$.  We defined the
quantity $D_n(f,\CC)$ for $n\in \N_0$ as the minimal cardinality
of a set of $n$-tiles for $(f,\CC)$ whose union is connected and
joins opposite sides of $\CC$ (see \eqref{def:dk} and
Definition~\ref{def:connectop}). The combinatorial expansion
factor $\Lambda_0(f)$\index{combinatorial expansion factor}
\index{L0@$\Lambda_0$} 
 of $f$ as discussed  in
Chapter~\ref{cha:combexpfac} is related to the growth rate of
$D_n(f,\CC)$, and given by
\begin{equation*}
  \Lambda_0(f)=\lim_{n\to \infty} D_n(f,\CC)^{1/n}. 
\end{equation*}
We have already seen that $\Lambda_0(f)\le \deg(f)^{1/2}$ if  $f$ has no periodic critical points; see the
discussion after Proposition~\ref{prop:Ahlforsreg}.  This
implies that $D_n(f,\CC)$ cannot grow much faster than
$\deg(f)^{n/2}$ as $n\to \infty$. As the following statement
shows, up to a multiplicative constant this is actually a
precise upper bound for the growth rate of $D_n(f,\CC)$.

\begin{prop}
  \label{prop:macgrdn}  
  \index{d0 n@$D_n$}
  \index{degf@$\deg(f)$}
  Suppose $f\: S^2\ra S^2$ is an expanding Thurston map,  and
  $\CC\sub S^2$ is a Jordan curve with $\post(f)\sub \CC$.
  Then there exists a constant $c>0$ such that
  \begin{equation}
    D_n(f,\CC)\le c \deg(f)^{n/2} 
  \end{equation} 
  for all $n\in \N$. Moreover, we have $\Lambda_0(f) \le
  \deg(f)^{1/2}$.  
\end{prop}
This will be proved in Section~\ref{sec:short-e-chains}. 
In view of Proposition~\ref{prop:macgrdn} one can ask whether
there are maps for which $D_n(f,\CC) $ actually grows as fast as
$\deg(f)^{n/2}$ as $n\to \infty$. It turns out that 
Latt\`es maps are essentially characterized by this property. This is the
main result of this chapter. 

\begin{theorem} 
  \label{thm:Qianthm} 
  Let $f\: S^2\ra S^2$ be an expanding Thurston map.  Then $f$
  is topologically conjugate to a 
  Latt\`es map\index{Latt\`{e}s map}
  if and only if the following conditions are
  true:
  
  \begin{enumerate}
  \item  
    \label{Qianthm1}  
    $f$ has no periodic critical points.
    
  \item   
    \label{Qianthm2}
    There exists $c>0$, and a Jordan curve $\CC\sub S^2$ with
    $\post(f)\sub \CC$ such that for all $n\in \N_0$ we have
    \begin{equation} 
      \label{eq:maxgrDn}
      D_n(f,\CC)\ge c \deg(f)^{n/2}.
    \end{equation} 
    \index{d0 n@$D_n$}
    \index{degf@$\deg(f)$}
  \end{enumerate} 
\end{theorem}
  
This theorem was proved by  Qian Yin as part of her thesis (the result  was later published in \cite{Qian15}).
The proof will occupy most of this
chapter. We will follow Yin's approach with some modifications
and simplifications.
  
If $f$ is a rational map, then we get a slightly stronger
statement.
    
\begin{cor} 
  \label{cor:Qianthm} 
  Let $f\: \CDach\ra \CDach$ be a rational expanding Thurston
  map. Then $f$ is a Latt\`es map if and only if condition
  \ref{Qianthm2} in Theorem~\ref{thm:Qianthm} is satisfied.
\end{cor}

We do not know whether a similar  improvement of Theorem~\ref{thm:Qianthm} is possible for arbitrary (not necessarily rational) expanding Thurston maps.    

If \eqref{eq:maxgrDn} is true for {\em some} Jordan curve $\CC$,
then the same condition holds for {\em each} Jordan curve
$\CC'\sub S^2$ with $\post(f)\sub \CC'$ (in general with a
different constant $c>0$ depending on $\CC'$). This follows
from the fact that $D_n(f,\CC)\asymp D_n(f,\CC')$ which was
shown in Lemma~\ref{lem:Dnprops} (see \eqref{DDtilde}).
  
By Proposition~\ref{prop:macgrdn} the inequality
$D_n(f,\CC)\lesssim \deg(f)^{n/2}$ is always true;  so
\eqref{eq:maxgrDn} says that $D_n(f,\CC) \asymp \deg(f)^{n/2}$
as $n\to \infty$. This asymptotic behavior of $D_n(f,\CC)$
implies that $\Lambda_0(f)= \deg(f)^{1/2}$. This equality is
slightly weaker than the requirement \eqref{eq:maxgrDn}. 

It
would be very interesting to characterize the expanding Thurston
maps whose combinatorial expansion factor is maximal in this
sense.  Besides for expanding Thurston that are topologically
conjugate to Latt\`es maps, it is also satisfied for certain
Latt\`es-type maps (similar to Example~\ref{ex:notattained}
where the associated linear map on $\R^2$ has a shear
component).

This chapter is organized as follows. In Section~\ref{sec:visual-metrics-that} we will formulate  criteria for an expanding Thurston map  to be topologically conjugate to a Latt\`es map in terms of the existence of  visual metrics with special properties (see
 Theorem~\ref{thm:visual_2Ahlfors_Lattes} and Corollary~\ref{cor:Lattes_L_degf}). Taking a technical lemma for granted (Lemma~\ref{lem:Nn_degf}), we will then derive Theorem~\ref{thm:Qianthm} and Corollary~\ref{cor:Qianthm} from this. 

The proof of Lemma~\ref{lem:Nn_degf} requires some preparation which is discussed in Section~\ref{sec:separating-sets-with}. The lemma is then proved in Section~\ref{sec:short-e-chains}, where we will also establish Proposition~\ref{prop:macgrdn}.   

\section{Visual metrics, $2$-regularity, and Latt\`{e}s maps}
\label{sec:visual-metrics-that}


The main implication (i.e., the sufficiency) in
Theorem~\ref{thm:Qianthm} 
will be a consequence of the
following statement.

\begin{theorem}
  \label{thm:visual_2Ahlfors_Lattes}
  \index{Latt\`{e}s map}
  \index{metric!visual}
  \index{visual metric}
  \index{Ahlfors regular}
  \index{$2$-regular}
  \index{r@$\varrho$}
  Let $f\colon S^2 \to S^2$ be an expanding Thurston map. Then
  $f$ is topologically conjugate to a Latt\`{e}s map if and only
  if there is a visual metric $\varrho$ for $f$ such that
  $(S^2, \varrho)$ is Ahlfors $2$-regular.
\end{theorem}

\begin{proof}
 First suppose that  $f$ is topologically conjugate to a
  Latt\`{e}s map. To see the  existence of a visual metric with the desired property, we can in fact assume that
  $f\colon \CDach \to\CDach$ is a Latt\`{e}s map. Then the
  canonical orbifold metric\index{canonical orbifold!metric}\index{metric!canonical orbifold}\index{o@$\omega$}\index{orbifold!canonical metric}  
$\varrho=\omega$ for $f$ is a visual metric
  with expansion factor $\Lambda = \deg(f)^{1/2}$ (see
  Proposition~\ref{prop:paravisorbmet}). In addition, $f$ has no
  periodic critical points (see
  Theorem~\ref{thm:Lattesstruc}~\ref{item:Lattessrucii}). By
  Proposition~\ref{prop:Ahlforsreg} the space $(\CDach, \varrho)$
  is Ahlfors $2$-regular.

  Conversely,   assume that $f\:S^2\ra S^2$ is an expanding Thurston map, 
  and $\varrho$ is a visual metric for $f$ such that $(S^2,\varrho)$ is
  Ahlfors $2$-regular.  Here the underlying measure is 
  $2$-dimensional Hausdorff measure $\mathcal{H}^2_{\varrho}$.
Since every Ahlfors regular space is
  doubling, it follows from
  Theorem~\ref{thm:S2vsf}~\ref{item:S2doubling} that $f$ has no
  periodic critical points.  
    
  The metric space $(S^2, \varrho)$ is linearly
  locally connected by
  Proposition~\ref{prop:annLLC}~\ref{item:visSphLLC}.  Hence
  Theorem~\ref{mainthm} applies and $(S^2, \varrho)$ is
  quasisymmetrically equivalent to the Riemann sphere $\CDach$
  equipped with the chordal metric $\sigma$. This in turn
  implies by Theorem~\ref{thm:S2vsf}~\ref{item:S2qsphere} that
  $f$ is topologically conjugate to a rational map. Since our
  assumptions are 
  preserved under such a conjugacy, we are reduced to the case
  where $f\colon \CDach \to \CDach$ is a {\em rational} expanding Thurston  map.

By Lemma~\ref{cor:visualqsmetric} the
  identity map $\id_{\CDach}\: (\CDach, \sigma)\ra (\CDach,
  \varrho)$ is a quasisymmetry.  Hence   Lebesgue measure $\leb$ on $\CDach$ and 
  $\mathcal{H}^2_\varrho$  are absolutely continuous with respect to
  each other by
  Proposition~\ref{prop:abscont2reg}. On the other hand, by
  Proposition~\ref{prop:Ahlforsreg} the measure of maximal
  entropy $\nu_f$ of $f$ and $\mathcal{H}^2_\varrho$ are
  comparable. Thus, $\nu_f$ is absolutely continuous with respect
  to $\leb$. Zdunik's theorem (i.e.,
  Theorem~\ref{thm:abscontimpLattes}) now implies that $f$ is a
  Latt\`es map.
\end{proof}

The previous theorem and its proof combined with  
Proposition~\ref{prop:Ahlforsreg} also give  the following statement.

\begin{cor}
  \label{cor:Lattes_L_degf}
  Let $f\colon S^2\to S^2$ be an expanding Thurston map. Then
  $f$ is topologically conjugate to a Latt\`{e}s map if and only
  if
  \begin{enumerate}
  \item $f$ has no periodic critical points, and
  
  \item there is a visual metric $\varrho$ for $f$ with
    expansion factor $\Lambda = \deg(f)^{1/2}$.  
  \end{enumerate}
\end{cor}

Thus in order to prove that conditions~\ref{Qianthm1} and \ref{Qianthm2}
in  Theorem~\ref{thm:Qianthm} imply that $f$ is topologically
conjugate to a Latt\`{e}s map, it is enough to construct a
visual metric 
$\varrho$ with expansion factor $\Lambda =\deg(f)^{1/2}$. 

We will construct such a metric from combinatorial data. To
this end, we fix a Jordan curve $\CC\sub S^2$ with $\post(f)\sub \CC$ and 
consider tiles for $(f,\CC)$. For $x,y\in S^2$  we then define 
\begin{equation}
  \label{eq:def_Nnxy}
  N_n(x,y)\coloneqq  \min\,\{  \length (P) : 
  P \text{ is a chain of $n$-tiles joining }
  x \text{ and } y\}.
\end{equation}
See Definition~\ref{def:chains} for our terminology. 
The following lemma will now provide the main step in the proof
of 
Theorem~\ref{thm:Qianthm}. 
\begin{lemma}
  \label{lem:Nn_degf}
  Let $f\colon S^2\to S^2$ be an expanding Thurston map that has
  an invariant Jordan curve $\CC\subset S^2$ with $\post(f)
  \subset \CC$. Suppose that  $\Lambda\coloneqq \deg(f)^{1/2}>2$ and that 
  condition \eqref{eq:maxgrDn} in  Theorem~\ref{thm:Qianthm} is
  satisfied for $\CC$. Then
  \begin{equation*}
    N_n(x,y) \asymp \Lambda^{n- m(x,y)} 
  \end{equation*}
   for all $x,y\in S^2$ and $n\in \N$ with 
  $n\geq m(x,y)+1$. Here the constant $C(\asymp)$ is independent of
  $x$, $y$, and  $n$. 
\end{lemma}
We use the notation  $m(x,y)=m_{f,\CC}(x,y)$, where 
$m_{f,\CC}$  is as  in
Definition~\ref{def:mxy}. Note  that 
$m(x,y)=\infty$ if $x=y$. In this case, the statement   is vacuous as there is no integer with   $n\geq m(x,y)=\infty$.

 The proof of this lemma will occupy the bulk of the next two
sections. The assumptions that $\CC$ is invariant and that
$\deg(f)^{1/2}>2$ are not essential  and  the lemma is
still true without these hypotheses, but  they will  simplify the proof.

The previous lemma gives us the following consequence.

\begin{lemma}
  \label{lem:Qian_cond_vis_ex}
  Let $f\colon S^2\to S^2$ be an expanding Thurston map that
  satisfies condition~\eqref{eq:maxgrDn} in 
  Theorem~\ref{thm:Qianthm}. Then there is a visual metric
  $\varrho$ for $f$ with expansion factor $\Lambda =
  \deg(f)^{1/2}$.
\end{lemma}

\begin{proof}
  Let $f\colon S^2\to S^2$ be an expanding Thurston map that
  satisfies~\eqref{eq:maxgrDn} for some Jordan curve $\CC\sub S^2$ with
  $\post(f) \subset \CC$. 
  
    We assume first that $\CC$ is $f$-invariant and
  $\Lambda \coloneqq   \deg(f)^{1/2}>2$. Then for $x,y\in S^2$ we define 
  \begin{equation}
    \label{eq:def_visual_Qian}
     \varrho(x,y)
     = 
     \limsup_{n\to\infty}  \Lambda^{-n} N_n(x,y). 
  \end{equation}
  Note that  $N_n(x,x)=1$ for  $x\in S^2$ and $n\in
  \N$, and so $\varrho(x,x)=0$. This together with  Lemma~\ref{lem:Nn_degf} shows that
  \begin{equation*}
    \varrho(x,y) \asymp \Lambda^{-m(x,y)},
  \end{equation*}
  where $m=m_{f,\CC}$ and $C(\asymp)$ is independent of $x$ and
  $y$. Note that in particular  $\varrho(x,y)\in [0,\infty)$ for $x,y\in S^2$.
    
So $\varrho$ will be a  visual
  metric for $f$ with expansion factor $\Lambda =
  \deg(f)^{1/2}$ if we can show that $\varrho$ is indeed a
  metric.  For $x,y\in S^2$ we obviously have
  $\varrho(x,y)=\varrho(y,x)$ and  
   $\varrho(x,y)=0$ if and only
  if $x=y$.  The triangle inequality for $\varrho$ follows
 immediately  from the inequality  
  \begin{equation*}
    N_n(x,z)\le N_n(x,y)+N_n(y,z)
  \end{equation*}
  valid for all $n\in \N_0$ and $x,y,z\in S^2$. Thus $\varrho$
  is a visual metric with expansion factor $\Lambda =
  \deg(f)^{1/2}$. 

 We now consider the general case without assuming that 
 $\CC$ is 
  $f$-invariant and $\deg(f)^{1/2}>2$. 
  By Theorem~\ref{thm:main} we can pick an iterate $F=f^k$ of
  $f$ that has an $F$-invariant Jordan curve $\CC\sub S^2$ with
  $\post(f)=\post(F)\sub \CC$. By picking $k$ large enough we may  also 
  assume that
  $\deg(F)^{1/2}>2$. 
 
  Recall that by \eqref{DDtilde} in Lemma~\ref{lem:Dnprops}
  condition 
  \eqref{eq:maxgrDn} is essentially independent of the chosen
  Jordan curve. So  we may assume that $f$
  satisfies this condition for
  the $F$-invariant curve $\CC$. Since the  $n$-tiles for
  $(F,\CC)$ are precisely  the $(nk)$-tiles for $(f,\CC)$ (see Proposition~\ref{prop:celldecomp}~\ref{item:celdecompiter}),
it
  follows that 
  \begin{equation*}
    D_{n}(F,\CC) 
    =
    D_{nk}(f,\CC) 
    \gtrsim 
    \deg(f)^{nk/2}=\deg(F)^{n/2}
  \end{equation*}
  for all $n\in \N$. So by the first part
  of the proof, there exists  a visual metric $\varrho$ for $F$ with
  expansion factor $\Lambda_F= \deg(F)^{1/2}$. 
  Proposition~\ref{prop:visualsummary}~\ref{item:vsfF} implies  
  that $\varrho$ is a visual metric for $f$ with
  expansion factor $\Lambda = \Lambda_F^{1/k} =
  \deg(F)^{1/(2k)}= \deg(f)^{1/2}$. 
\end{proof}

Assuming that Lemma~\ref{lem:Nn_degf} is valid, we can now  
finish the proof of Theorem~\ref{thm:Qianthm}. 

\begin{proof}
  [Proof of Theorem~\ref{thm:Qianthm}]
 The sufficiency part of  
  Theorem~\ref{thm:Qianthm} is now easy. Indeed,  let $f\colon S^2 \to
  S^2$ be an expanding Thurston map without periodic critical
  points,  
and $\CC\subset S^2$ be a Jordan curve 
with $\post(f)
  \subset \CC$ such that \eqref{eq:maxgrDn} is satisfied.
  By Lemma~\ref{lem:Qian_cond_vis_ex} there is a visual metric
  $\varrho$ for $f$ with expansion factor $\Lambda =
  \deg(f)^{1/2}$. Corollary~\ref{cor:Lattes_L_degf} shows that
  $f$ is topologically conjugate to a Latt\`{e}s map as desired.  

  To prove the reverse implication, let $f\colon
  S^2\to S^2$ be topologically conjugate to a Latt\`{e}s map. We
  want to show  that $f$ 
  satisfies the con\-di\-tions \ref{Qianthm1} and \ref{Qianthm2}
  in Theorem~\ref{thm:Qianthm}.  Since these conditions are
  invariant under topological conjugacy (in a suitable sense),
  we may assume that $f$ is a Latt\`es map on the Riemann sphere
  $\CDach$. Then $f$ has no periodic critical points
  (see Theorem~\ref{thm:Lattesstruc}~\ref{item:Lattessrucii}),
  and so \ref{Qianthm1} is true.
 
 We also know that $f$ has a parabolic orbifold. Let $\om$ be the
 canonical orbifold metric of $f$ on $\CDach$ (see
 Section~\ref{sec:expratThmaps}). Then by
 Proposition~\ref{prop:paravisorbmet} this metric is a visual
 metric for $f$ with expansion factor $\Lambda=\deg(f)^{1/2}$.


 Now we pick a Jordan curve $\CC\sub \CDach$ with $\post(f)\sub
 \CC$ and consider tiles for $(f,\CC)$. Fix $n\in \N$. Then,
 according to the definition of $D_n(f,\CC)$, we can find a
 connected union $K$ of $n$-tiles that joins opposite sides of
 $\CC$ and consists precisely of $D_n(f,\CC)$  tiles. Then
 $\diam_\om(K)\ge \delta_0$, where $\delta_0>0$ is the constant
 in \eqref{defdelta} for the underlying base metric $\omega$ on
 $\CDach$.  Among the $n$-tiles that form $K$ we can find a
 chain $X_1, \dots, X_N$ of distinct $n$-tiles that joins two
 points $x,y\in K$ with $\om(x,y)=\diam_\om(K)$.  Then $N\le
 D_n(f,\CC)$, and
$$ \sum_{i=1}^N \diam_\om(X_i)\ge \om(x,y)\ge \delta_0. $$ 
Since $\om$ is a visual metric for $f$ with expansion factor
$\Lambda=\deg(f)^{1/2}$,  by
Proposition~\ref{lem:expoexp}~\ref{item:expoex2}  we have $\diam_\om(X_i)\lesssim 
\Lambda^{-n}$ with  $C(\lesssim)$ independent $X_i$ and $n$. Hence
$N\deg(f)^{-n/2}\gtrsim  \delta_0$, and so 
$$ D_n(f,\CC) \ge  N\gtrsim \deg(f)^{n/2}, $$
where $C(\gtrsim)$ is independent of $n$. This shows that condition \ref{Qianthm2}
in Theorem~\ref{thm:Qianthm} is also true. 
\end{proof}

Corollary~\ref{cor:Qianthm} is an immediate consequence of Theorem~\ref{thm:Qianthm}.

\begin{proof}[Proof of Corollary~\ref{cor:Qianthm}] The ``only if" direction of the statement follows from Theorem~\ref{thm:Qianthm}. 

Conversely, suppose that $f$ is a rational expanding Thurston map 
satisfying condition \ref{Qianthm2} in Theorem~\ref{thm:Qianthm}.
Then $f$ has no periodic critical points (Proposition~\ref{prop:rationalexpch}), and so $f$ is topologically conjugate to a 
Latt\`es map by  Theorem~\ref{thm:Qianthm}. Then $f$ has a parabolic orbifold as follows from 
Proposition~\ref{prop:Thequivsamesig}. Hence $f$ is itself a Latt\`es map by Theorem~\ref{thm:Lattesstruc}~\ref{item:Lattessrucii}. \end{proof}

\section{Separating sets with tiles}
\label{sec:separating-sets-with}

In the previous section we have seen that in order to prove
Theorem~\ref{thm:Qianthm} we have  to establish Lemma~\ref{lem:Nn_degf}. This means that for given  $x,y\in S^2$ 
we have to estimate $N_n(x,y)$, the minimal number of $n$-tiles in a chain 
joining $x$ and $y$,  in terms of
$m_{f,\CC}(x,y)$ (see \eqref{eq:def_Nnxy}). Here we are of course
assuming that \eqref{eq:maxgrDn} holds.
The lower bound easily follows from \eqref{eq:maxgrDn}
together with Lemma~\ref{lem:flowerbds}.

\begin{lemma}
  \label{lem:Nn_lower_bd}
  Let $f\colon S^2\to S^2$ be an expanding Thurston map 
that has   an invariant Jordan curve $\CC\subset S^2$ with
  $\post(f)\subset \CC$. Suppose that  \eqref{eq:maxgrDn} holds and define 
  $\Lambda\coloneqq \deg(f)^{1/2}$. Then
  \begin{equation*}
    N_n(x,y) \gtrsim \Lambda^{n - m(x,y)},
  \end{equation*}
  for all  
   $x,y\in S^2$ and $n\in \N$ with 
   $n \geq
  m(x,y)+1$. 
  Here the constant $C(\gtrsim)$ is independent of $x$, $y$, and $n$.  
\end{lemma}

Here again $m(x,y)=m_{f,\CC}(x,y)$ and it is 
 understood that we use $n$-tiles for $(f,\CC)$ in the definition $N_n(x,y)$. 

\begin{proof} Let   $x,y\in S^2$ be  two distinct points and   
 $m\coloneqq m(x,y)$. We pick $(m+1)$-tiles $X^{m+1}$ and
  $Y^{m+1}$  (for $(f,\CC)$)    with $x\in X^{m+1}$ and $y\in
  Y^{m+1}$. Then $X^{m+1}\cap Y^{m+1}=\emptyset$ by definition
  of $m(x,y)$. 

  If $n\ge m+1$ and $P$ is an arbitrary chain of
  $n$-tiles joining $x$ and $y$, then $P$ also joins $X^{m+1}$
  and $Y^{m+1}$. Since these are two disjoint $(m+1)$-cells, it
  follows from Lemma~\ref{lem:flowerbds} that
  \begin{equation*}
    \length(P)\ge D_{n-m-1}(f,\CC).
  \end{equation*}
  So by using  assumption \eqref{eq:maxgrDn},
  we obtain
  \begin{equation*}
    N_n(x,y)\ge D_{n-m-1}(f,\CC) \gtrsim
    \Lambda^{n-m-1}\asymp \Lambda^{n-m}. \qedhere
  \end{equation*}
\end{proof}

We have to prove an inequality in the other
direction and  show that two points $x,y\in S^2$  can be joined by a
rather  short chain of $n$-tiles. For this we
use a duality argument that will give the existence of short
chains provided certain separating sets of tiles do not have too
small cardinality.  The key ingredient of this duality argument
is a well-known graph-theoretical statement, namely Menger's
theorem. Before we formulate this result, we first record some
definitions.
 
We consider a finite graph $G$. Here we take the combinatorial
point of view; so $G$ is just a pair $(V,E)$, where $V$ is a
finite set called the {\em set of vertices} of $G$ and $E\sub
V\times V$ is a subset of the set of pairs in $V$,
called the {\em set of edges} of $G$. We assume that $E$ is symmetric (i.e., $(x,y) \in E$ if and only if $(y,x)\in E$),  and disjoint from the diagonal $\{(x,x): x\in V\}$. 
If $(x,y)\in E$, we say that $x$ and $y$
are {\em joined by an edge} (in $G$).  Since  $E$ is symmetric,
we consider edges as non-oriented.
 
A {\em path} $P$ in $G$ is a finite sequence $x_1, \dots, x_n$ of
 vertices in $G$ so that successive vertices are distinct and 
joined by an edge. Such a path is said to \emph{join} $x_1$ and
$x_n$. 
The number $n\in \N$ is the {\em length} of $P$.  If a path
$P'$ 
can be obtained by deleting 
some of the vertices
of the sequence $x_1, \dots, x_n$, 
then it is called 
a {\em subpath} of $P$.   The path $P$ is {\em simple} if  all of its vertices are distinct.  If $A,B\sub V$, then a path in $G$ {\em joins} $A$ and
$B$, if the first vertex of the path lies in $A$, and its last
vertex lies in $B$. A path joining $A$ and $B$ in $G$ is called
an {\em $A$-$B$-path}.  A set $A\sub V$ is {\em connected} if for
all $a,a'\in A$ there exists a path in $A$ joining $a$ and
$a'$. A set $K\sub V$ {\em separates} $A$ and $B$ if every
$A$-$B$-path contains an element in $K$.  Given these
definitions, we have the following statement.

\begin{theorem}[Menger's theorem]
  \label{thm:Menger}
  \index{Menger's theorem}
  \index{max-flow min-cut theorem}
  Let $G$ be a finite graph with vertex set $V$, and $A,B\sub
  V$.  Then the minimal cardinality of a set separating $A$ and
  $B$ in $G$ is equal to the maximal number of pairwise disjoint
  $A$-$B$-paths in $G$.
\end{theorem}
   
For the proof  and more background see \cite[Section~3.3]{Di}.
This theorem can be seen as a special case of the max-flow min-cut theorem (see \cite[Section~6.2]{Di}). 

Suppose $M\in \N$ is a lower bound
for the cardinality of a set separating $A$ and $B$ in Theorem~\ref{thm:Menger}.  
Then, by  passing to subpaths if necessary,  we obtain at least
$M$ simple  $A$-$B$-paths  
that are pairwise disjoint. 
This implies that one of them must have length $\le \#V/M$. So a lower bound for the cardinality of a separating set leads to the existence of an 
$A$-$B$-path with controlled length.    
  
The graphs $G$ that we will consider in our context will have
sets of tiles as vertex sets. 
Then a tile  may be viewed in two ways: as a vertex in $G$, or as
a closed Jordan region in the underlying $2$-sphere.
In this situation we want to use
the topology of the sphere to obtain information
on separating sets in the graph. 
For this we will invoke some 
 well-known topological facts related to Janiszewski's lemma 
 (see Lemma~\ref{lem:Janiszewski}). Actually, we will require a more refined version stated in Lemma~\ref{lem:ABKsep}. Since its proof   is somewhat long and technical, and would distract from our present considerations, we included it in the appendix (see Section~\ref{sec:Janis}).

Let us  return to expanding Thurston maps. In the following, we fix such a map  $f\:S^2\ra S^2$ and 
 a Jordan curve $\CC\subset S^2$ with $\post(f)\subset
\CC$. All cells considered below  are  for $(f,\CC)$. 

Given $k\in \N_0$, we consider a simple $e$-chain of
$k$-tiles. Recall from Definition~\ref{def:e-chain} that this is
a sequence 
 $Z_1,\dots, Z_N$ of distinct $k$-tiles, where
$N\in \N$, and we suppose that there exist $k$-edges $E_1,
\dots, E_{N-1}$ with $E_i\sub Z_{i}\cap Z_{i+1}$ for $i=1, \dots ,
N-1$. 
We now consider  the ``interior'' of this $e$-chain, i.e., the set
\begin{equation} \label{defOmega}
  \Om\coloneqq  \bigcup_{i=1}^N \inte(Z_i)\cup \bigcup_{i=1}^{N-1}\inte(E_i).
\end{equation} 
Note that in general, $\Omega$ is not the interior of the underlying set
$\bigcup_{i=1}^N Z_i$; see Figure~\ref{fig:Om_AB} for an example. 

\begin{lemma}
  \label{lem:Om_simply_conn}
  \index{e-chain@$e$-chain}\index{chain!$e$-}
  The set   $\Omega$ defined  in \eqref{defOmega} is a simply connected region  in
  $S^2$. 
\end{lemma}


\begin{proof}
  For each $i=1, \dots, N-1$ the  set
 $$U_i\coloneqq \inte(Z_{i})\cup \inte(E_{i}) \cup \inte(Z_{i+1})$$
 is an open region (see Lemma~\ref{lem:specprop}~\ref{item:prop_cell4}). This
 implies that $\Om$ is open and connected, and hence a region.
 
 The additional statement that $\Om$ is simply connected follows
 from the fact that every loop $\ga$ in $\Om$ can be homotoped
 to a constant loop inside $\Om$.  To see this, define
 \begin{equation*}
   \Om_l\coloneqq  \bigcup_{i=1}^l \inte(Z_i)
   \cup 
   \bigcup_{i=1}^{l}\inte(E_{i})\sub \Om
 \end{equation*}
 for $l=1, \dots, N$, where we set $E_{N}=\emptyset$.  For
 each $i=1, \dots , N-1$ there exists a deformation retraction of
 the set $$\inte(E_{i})\cup \inte(Z_{i+1})\cup \inte(E_{i+1})$$ onto
 $\inte(E_{i})$ which implies that there exists a deformation
 retraction of $\Om_{i+1}$ onto $\Om_{i}$.  Here it is important
 that
 \begin{equation*}
   (\inte(Z_{i+1}) \cup\inte(E_{i+1}))\cap \Om_{i}=\emptyset
 \end{equation*}
 which follows from the facts that the $k$-tiles $Z_1, \dots,
 Z_N$ are all distinct and that the set $\inte(E_i)$ only meets
 the $k$-tiles $Z_{i}$ and $Z_{i+1}$ for $i=1, \dots, N-1$ (this
 follows from Lemma~\ref{lem:specprop}~\ref{item:prop_cell4}).

 By using these deformation retractions successively, every
 closed loop in $\Om=\Om_N$ can be homotoped inside $\Om$ to a
 loop in $\Om_1\sub \Om$. The set $\Om_1$ is the union of an
 open Jordan region  with an open arc on its  boundary  if $N>1$
 and an open Jordan region if $N=1$. Hence $\Om_1$ is
 contractible, and the simple connectivity of $\Om$ follows.
\end{proof}

We now assume in addition that the Jordan curve $\CC$ is
$f$-invariant. As before, $Z_1, \dots ,Z_N$ 
denotes a simple $e$-chain
of $k$-tiles. 

Fix  $n\in \N_0$ with $n\ge k$.
Since $\CC$ is $f$-invariant, the
$k$-cells are subdivided into $n$-cells. We form a graph $G$
whose vertex set $V$ consists of all $n$-tiles contained in any
of the $k$-tiles $Z_i$. Then  $V$ is the set of all
$n$-tiles $X$ with $\inte(X)\sub 
\Om$. In $G$ we join two distinct vertices in $V$ given by
$n$-tiles $X$ and $Y$ by an edge if there exists an $n$-edge
$e\sub X\cap Y$ with $\inte(e)\sub \Om$.
  
Note that if $X,Y\in V$, $X\ne Y$, and $X$ and $Y$ share an
$n$-edge $e$, then there are two possibilities. Namely, $X$ and
$Y$ may lie in the same $k$-tile $Z_i$. Then necessarily
$\inte(e)\sub \inte (Z_i)\sub \Om$ and so the vertices $X$ and
$Y$ are joined by an edge in $G$.  If $X$ and $Y$ lie in
different $k$-tiles $Z_i$, then the only situation where $X$ and
$Y$ are joined by an edge in $G$ is when $X$ and $Y$ lie in
consecutive $k$-tiles of the $e$-chain, say $X\sub Z_{i}$ and
$Y\sub Z_{i+1}$, and $\inte(e)\sub \inte(E_i)\sub \Om$.  So we
only join ``across'' the $k$-edges $E_1, \dots, E_{N-1}$, but not
across any other $k$-edge contained in $Z_1\cup \dots \cup Z_N$.
   
Given a set of $n$-tiles $M$ we use the notation $|M|$ for the
underlying point set, i.e.,
 $$ |M|\coloneqq  \bigcup_{X\in M}X. $$
The next lemma relates connectedness and separation properties of sets of $n$-tiles considered as sets of vertices in $G$ with the 
corresponding properties of the underlying sets.

\begin{lemma} 
  \label{lem:G_Om} 
  \index{e-chain@$e$-chain}
  \index{chain!$e$-}
  \index{e-connected@$e$-connected}
  With the given assumptions, the following statements are
  true: 
 
 \begin{enumerate}
 \item 
   \label{item:grsep1a} 
   Let $\ga$ be a path in $\Omega$ and
   $M\sub V$ be the set of all $n$-tiles $X$ with  $X\cap \ga\ne \emptyset$.  Then $M$ is connected in
   $G$, and $M'=|M| \cap \Om$ is a connected subset of $\Om$.
 \item 
   \label{item:grsep1b}
   Let $\ga$ be a path in the boundary $\partial Z_i$ of one of the $k$-tiles $Z_i$, 
    and $M\sub V$ be the set of all $n$-tiles $X$
   with $X\sub Z_i$ and $X\cap \ga\ne \emptyset$. Then $M$ is
   connected in $G$, and $M'=|M| \cap \Om$ is a connected subset of
   $\Om$.
 
 \item 
   \label{item:grsep2} 
   If $A,B,K\sub V$, then $K$ separates $A$ and $B$ in $G$ if
   and only if $K'=|K|\cap \Om$ separates $A'=|A|\cap \Om$ and
   $B'=|B|\cap \Om$ in $\Om$. 
\end{enumerate}
\end{lemma} 

Here we say that the set $K'$ {\em separates} $A'$ and $B'$ in $\Om$ if every path $\ga$ in $\Om$ 
joining $A'$ and $B'$ meets $K'$ (see Section~\ref{sec:Janis}
for a more detailed discussion). 


\begin{proof}
 The paths in $G$ correspond precisely to $e$-chains
consisting of
  $n$-tiles $X_1, \dots, X_m\in V$ for which there exist
  $n$-edges $e_1, \dots, e_{m-1}$ with $e_i\sub X_{i}\cap X_{i+1}$
  and $\inte(e_i)\sub \Om$ for $i=1, \dots, m-1$, where $m\in
  \N$. For the rest of the proof  we call
  such an $e$-chain of $n$-tiles simply an  {\em admissible $e$-chain}. If $M\sub V$ is a given set and  $X_1, \dots , X_m\in M$, then we say it is an  admissible $e$-chain
  {\em in} $M$. 
The admissible
  $e$-chain  $X_1, \dots,X_m$ 
\emph{joins} two $n$-tiles $X, Y$ if $X=X_1$ and $Y=X_m$ and
two points 
$p,q\in \overline \Om$ if
  $p,q\in \bigcup_{i=1}^m X_i$.

\smallskip 
  \ref{item:grsep1a}    The
  argument for this is along the lines of the proof of
  Lem\-ma~\ref{lem:echain} with small modifications.  We first  analyze connectivity properties of admissible $e$-chains near points in $\Om$.
 
  For fixed $p\in \Omega$, we define 
  \begin{equation*}
    \Om^n(p)= \bigcup\{\inte{(c)} 
    : c \text{ is an $n$-cell with } p\in c\}. 
  \end{equation*}
  Recall that by definition $\inte{(c)}=c$ if $c$ is a
  $0$-dimensional $n$-cell (i.e., $c$ is a singleton set
  consisting of an $n$-vertex).  We know that each point in $S^2$
  is contained in the interior of a unique $n$-cell (see
  Lemma~\ref{lem:uniondisjint}). Thus, there are three types of
  sets $\Om^n(p)$ depending on whether $p$ is contained in the
  interior of an $n$-tile $X_p$, the interior of an $n$-edge
  $e_p$, or is an $n$-vertex.

  In the first case, $\Om^n(p)= \inte (X_p)$. In the second case,
  $\Om^n(p)$ is the union of $\inte(e_p)$ with the interiors of
  the two $n$-tiles that contain $e_p$ (see
  Lemma~\ref{lem:specprop}~\ref{item:prop_cell4}). In the third
  case, $\Om^n(p)$ is the $n$-flower of $p$ (see
  Definition~\ref{def:flower} and
  Lemma~\ref{lem:flowerprop}~\ref{item:flower_prop1}). In particular, 
  $\Om^n(p)$ is always an open set.

\smallskip
 {\em Claim 1.} $\Om^n(p)\subset \Omega$.
 
\smallskip 
 Recall from Lemma~\ref{lem:mincell} that the interior of each 
  $n$-cell is contained in the interior of a unique
  $k$-cell. Since $\Omega$ is a union of interiors of $k$-cells, we have 
   $\inte(c)\subset \Omega$ for an $n$-cell $c$ if
  and only if $\inte(c)$ contains a point in $\Omega$.
 Now  every $n$-cell $c$ with $p\in c$ has points in
  its interior arbitrary close to $p$. Since $p\in \Om$ and $\Omega$ is open, we conclude that $\inte(c)\sub \Om$ for such a cell $c$.  Claim~1 follows.
 
\smallskip 
The same argument also shows that $\inte(X)
  \subset \Omega$   for every $n$-tile $X$ with $p\in X$;  in this case,   $X$ represents a vertex in the graph $G$.
  
   Now  let $M\sub V$ be the set associated to a path $\ga$ as in the
  statement.

   \smallskip 
  {\em Claim 2.} If $x,y\in \overline \Om^n(p)$ and $p\in \ga$, then there exists an
    admissible $e$-chain in $M$ joining $x$ and $y$.  
    
     \smallskip 
 Note that $\Om^n(p)\sub \Om$ and all $n$-tiles $X$ with $p\in X$ belong to $M$.  
 So if $\Om^n(p)$ is of the first type, then 
   $X_p$ forms an  admissible
  $e$-chain  in $M$  joining $x,y\in  \overline \Om^n(p)=X_p$. For the second
  type, the two $n$-tiles containing $e_p$ form such an
  $e$-chain. Finally, for the third type the $n$-tiles
  $X_1, \dots,X_d$ containing $p$ labeled cyclically around $p$
  form such an  $e$-chain (see
  Lemma~\ref{lem:specprop}~\ref{item:prop_cell5}). 
  Claim~2 follows.

\smallskip
     The proof that $M$ is connected in $G$ amounts to
  showing that every two $n$-tiles in $M$ can be joined by an
admissible   $e$-chain  in $M$.
 For this in turn it is enough to assume that $\ga\: [a,b]\ra \Om$ is defined on a compact interval $[a,b]\sub\R$ and 
   show that if $X,Y\in V$, and $\ga(a)\in X$, $\ga(b)\in Y$,
  then there is an admissible $e$-chain in $M$  joining $X$
  and $Y$. 
  
 In the following, we will show that there exists  an admissible
 $e$-chain $X_1, \dots, X_m$ in $M$ with $\ga(a)\in X_1$ 
and $\ga(b)\in X_m$. 
Then Claim~2  (applied to $\overline \Om^n(\ga(a))$ 
and points $x\in \inte(X)$, $y\in \inte(X_1)$)
implies 
   that there exists an admissible $e$-chain in $M$ joining $X$ and $X_1$. Similarly, we can find such a chain joining $X_m$ and $Y$. A suitable concatenation of chains will then produce 
  an admissible $e$-chain in $M$   that joins $X$ and $Y$.
  
 To find the required chain for  $\ga(a)$ and $\ga(b)$,    
 let  $I\subset [a,b]$ be the set of all numbers
  $s\in[a,b]$ such that there exists an admissible $e$-chain in $M$
  joining $\gamma(a)$ and $\gamma(s)$. Clearly $a\in I$,
  and so  $I$ is non-empty. As in the proof of
  Lemma~\ref{lem:echain}, one shows that $I$ is closed.
  On the other hand,  $I$ is open in $[a,b]$ as follows from
  Claim~2 together with the fact that $\Om^n(p)$ for each $p\in \Om$ is open. So 
  $I=[a,b]$, and it follows that $M$ is connected in $G$. 

    The connectivity of the set $M'=|M|\cap \Om$ can easily be
  derived from this. Indeed, to join  two points $p,q\in M'$,
  one chooses $n$-tiles $X,Y\in M$ with $p\in X$ and $q\in
  Y$. By what we have just seen,
one can find an
  admissible $e$-chain consisting of $n$-tiles $X_1=X, \dots, X_m=Y$
  in $M$.  There are $n$-edges $e_1, \dots, e_{m-1}$ with $e_i\sub
  X_{i}\cap X_{i+1}$ and $\inte(e_i)\sub \Om$ for $i=1, \dots,
  m-1$. 
  
  To find a path $\ga$ that stays in $|M|\cap \Om$ and joins
  $p$ and $q$, we travel from $p$ to a point in $\inte(e_1)$
  along an arc whose interior stays in $\inte(X_1)$. We cross
  over to the interior of the next tile $X_2$ and travel along
  an arc whose interior stays in $\inte(X_2)$ to a point in
  $\inte(e_2)$, etc., until we finally join a suitable point in
  $\inte(e_{m-1})$ to $q$ by an arc whose interior stays in
  $\inte(X_m)$. The concatenation of these arcs gives a path in
  $M'=|M|\cap \Om$ joining $p$ and $q$.
 
  \smallskip 
  \ref{item:grsep1b}  
  Let $\gamma$ be a path  in the boundary of one of the
  $n$-tiles $Z_i$.  Similar to the previous argument, for 
  $p\in \partial Z_i$ 
   we define
  \begin{equation*}
    \Om^n(p) = \bigcup\{ \inte(c) 
    : c \text{ is an $n$-cell with } p\in c 
    \text{ and }
    c\subset Z_i \}. 
  \end{equation*}
  If $p$ is contained in the interior of an $n$-edge $e_p\subset
  \partial Z_i$, then  $\Om^n(p)$ is the union of $\inte(
  e_p)$ with the interior of the unique $n$-tile that contains
  $e_p$ and is contained in $Z_i$. If $p$ is an $n$-vertex, there
  are exactly two
  $n$-edges belonging to the cycle of $p$ and contained in
  $\partial Z_i$.  This implies that the Jordan curve $\partial
  Z_i$ splits the $n$-tiles and the $n$-edges in the
  cycle of $p$ into two families,  
  namely those with interior contained in $Z_i$ and those with
  interior disjoint from $Z_i$. Moreover, the interiors of the
  $n$-tiles and $n$-edges in the first family are all contained
  in $\Om^n(p)$. 

  In any case,  $\Om^n(p)$ is again a relatively  open neighborhood of $p$ in
  $Z_i$. If $M$ is defined as in the statement and $p\in \ga\sub \partial Z_i$, then 
 any two points in $\overline \Om^n(p)\sub  Z_i$ can be
  connected by an admissible $e$-chain in $M$. The argument is now
  identical to the one in case \ref{item:grsep1a}. 

  \smallskip 
  \ref{item:grsep2} ``$\Leftarrow$'' 
  With the given setup suppose that $K'=|K|\cap \Om$ separates
  $A'=|A|\cap \Om $ and $B'=|B|\cap \Om$ in $\Om$. A path in $G$
  joining $A$ and $B$ corresponds to  an admissible $e$-chain $P$
  consisting of $n$-tiles $X_1, \dots, X_m\in V$ with $X_1\in A$
  and $X_m\in B$. Then there exist $n$-edges $e_1, \dots,
  e_{m-1}$ with $e_i\sub X_{i}\cap X_{i+1}$ and $\inte(e_i)\sub \Om$
  for $i=1, \dots, m-1 $.  We have to show that $P\cap K\ne
  \emptyset$.

  Similar to the last part of the proof of
  \ref{item:grsep1a}, one can find a path $\ga$ that lies in the
  set
  \begin{equation*}
    \bigcup_{i=1}^m \inte(X_i)
    \cup
    \bigcup_{i=1}^{m-1}\inte(e_i)\sub \Om,
  \end{equation*}
  and has one endpoint in $\inte(X_1)$ and one in $\inte(X_m)$.
  Then $\ga$ does not meet any tile in $V$ that does not belong
  to $P$. The path $\ga$ joins $A'$ and $B'$ in $\Om$ and hence
  meets $K'\sub |K|$ by our assumptions. By choice of $\ga$ we
  then must have $P\cap K\ne \emptyset$ as desired.
  
  \smallskip 
  \ref{item:grsep2} ``$\Rightarrow$''   
  Suppose $K$ separates $A$ and $B$ in $G$, and suppose $\ga$ is
  a path in $\Om$ joining $A'=|A|\cap \Om$ and $B'=|B|\cap
  \Om$. Let $M$ be the set of all $n$-tiles that meet
  $\ga$. Since $\ga\sub\Om$ we have $M\sub V$. By
  \ref{item:grsep1a} the set $M$ is connected in $G$, and so
  there exists a path $P$ in $G$ that starts in $A$ and ends in
  $B$ and consists of $n$-tiles in $M\sub V$. By our assumptions
  we have $P\cap K\ne \emptyset$. In particular, $M$ and $K$
  have a tile in common which implies that $\ga \cap K' \ne
  \emptyset$, where $K'=|K|\cap \Om$.  Hence $K'$ separates $A'$
  and $B'$ in $\Om$.
\end{proof}

The following lemma provides the main estimate of this section. We will use the 
quantity $D_n(f,\CC)$  to estimate the minimal number of $n$-tiles that are
required to separate certain sets.

The setting is still the same as in the previous lemma. So   we
are 
given a simple $e$-chain 
$Z_1, \dots, Z_N$ of $k$-tiles and
$\Omega$ is as in \eqref{defOmega}. The $k$-tiles $Z_i$ are
subdivided into $n$-tiles, and $G$ is the graph defined above with its  
vertex set $V$ given by 
the set of all $n$-tiles contained in one of the  tiles   $Z_i$.

\begin{lemma}
  \label{lem:size_separate}
  Let $A,B\sub V$, and suppose that each of the sets $|A|$ and
  $|B|$ contains at least two distinct $k$-vertices and that $A'=|A|\cap
  \Om$ and $B'=|B|\cap \Om$ are connected. Let $K\sub V$ be a
  set that separates $A$ and $B$ in
  $G$. 
  Then 
  \begin{equation}
    \label{eqCEL:KDn}\#K\ge D_{n-k}(f,\CC) . 
  \end{equation}
\end{lemma}

\begin{proof}  Let $K\sub V$ be a set of minimal cardinality that separates
   the given sets $A$ and $B$.
   The
  existence of such a set $K$ follows from the fact that the whole
  vertex set $V$ separates $A$ and $B$ according to our
  definition. The setup
  is
illustrated in Figure~\ref{fig:Om_AB}. Note that the dashed
edges are not part of the region $\Omega$. 
The main idea of the proof is now to show that  $|K|$ is a connected set
not contained in a $k$-flower. This will lead to the desired bound.  

\begin{figure} 
  \centering
  \begin{overpic} 
    [width=12cm, tics=10, 
    ]{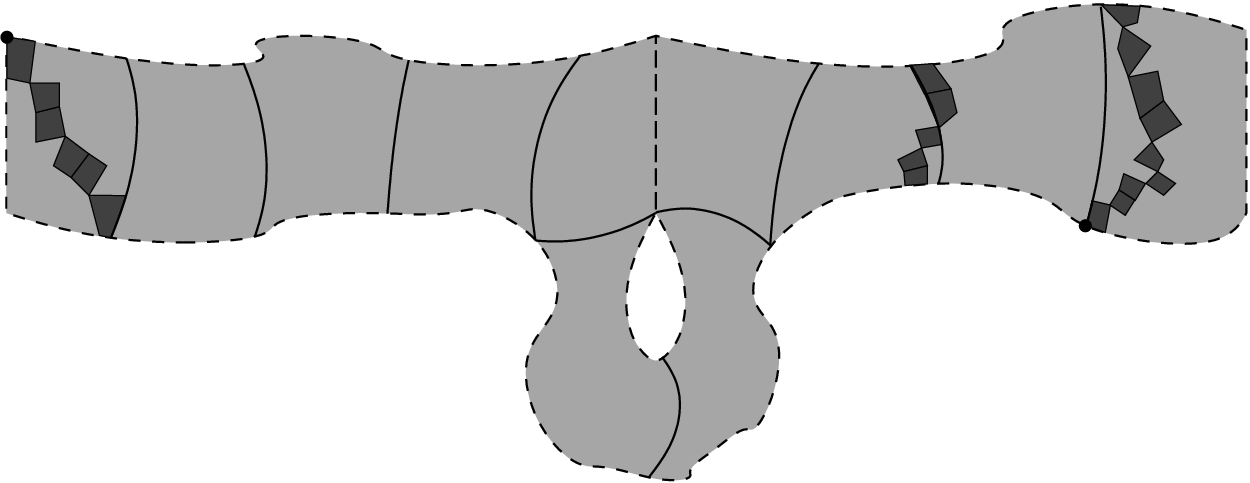}
    \put(-2.5,35){$a$} 
    \put(6.5,28){$A$}
    \put(2,16){$Z_1$} 
    \put(14,15){$Z_2$} 
    \put(9.5,35){$E_1$}
    \put(57,36){$Z_{i+1}$} 
    \put(62,14){$Z_{i}$}
    \put(56,23){$E_i$}
    \put(72,35){$K$}
    \put(85.5,17){$b$} 
    \put(94,25){$B$}
    \put(94,15){$Z_N$}
    \put(88,39){$E_{N-1}$}
  \end{overpic}
  \caption{The setup in Lemma~\ref{lem:size_separate}.}
  \label{fig:Om_AB}
\end{figure}
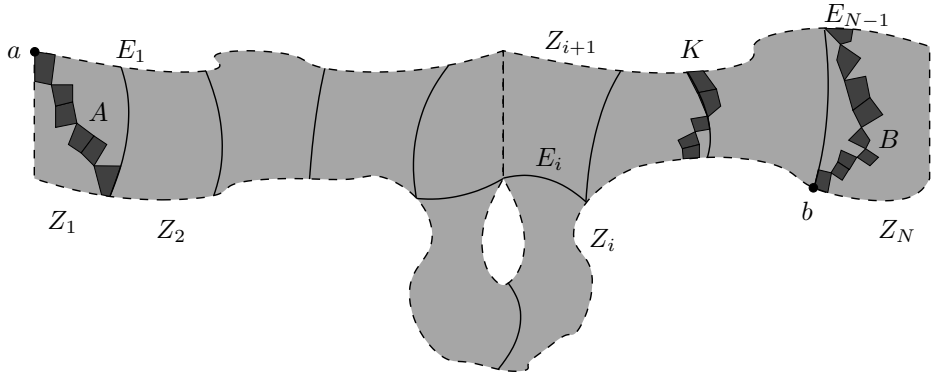

By Lemma~\ref{lem:G_Om}~\ref{item:grsep2} the
  set $K'=|K|\cap \Om$ separates $A'=|A|\cap \Om$ and
  $B'=|B|\cap \Om$ in $\Om$. Since $\Omega$ is a simply
  connected region by Lemma~\ref{lem:Om_simply_conn}, we can invoke 
  Lemma~\ref{lem:ABKsep} to find a component of $K'$
  that separates $A'$ and $B'$ in $\Om$. Since the interiors of tiles
  in $V$ are connected subsets of $\Om$, this component of $K'$
  is of the form $L'=|L|\cap \Om$ with $L\sub K$. Using 
  Lemma~\ref{lem:G_Om}~\ref{item:grsep2} again, we see that the set $L$
  separates $A$ and $B$ in $G$. Hence $L=K$ by the minimality of
  $K$. 
It follows that 
$K'=L'=|K|\cap \Om$ is connected, and so the same is  true for the set 
  $|K|\sub \overline {K'}$.

   We will now show that $|K|$ is not contained in a $k$-flower. We
  argue by contradiction, and assume that $|K|\sub W^k(p)$ for
  some $k$-vertex $p$. Each of the sets $|A|$ and $|B|$ contains
  at least two distinct $k$-vertices.  So we can pick $k$-vertices $a\in
  |A|$ and $b\in |B|$ with $a,b\ne p$. There are corresponding
  $n$-tiles $X\in A$ and $Y\in B$ with $a\in X$ and $b\in
  Y$. Since $X,Y\in V$, each of these $n$-tiles must be
  contained in one of the $k$-tiles $Z_i$. To keep the notation
  in the following argument simple, let us assume that $X\sub
  Z_1$ and $Y\sub Z_N$ (the general case requires inessential
  modifications).  Then $a\in \partial Z_1$ and $b\in \partial
  Z_N$.

  By the same argument as in the proof of Lemma~\ref{lem:nocut},
  we can now find a path $\alpha$ in the (topological) graph
  $G'=\partial Z_1 \cup \dots \cup \partial Z_N$ that joins $a$
  and $b$ and avoids $p$.  Namely, we can join $a$ to one of the
  endpoints of the $k$-edge $E_1\sub Z_1\cap Z_2$ by a (possibly
  degenerate) path $\alpha_1\sub \partial Z_1$ consisting of
  $k$-edges that avoid $p$. Then we join the endpoint of
  $\alpha_1$ in $E_1$ to one of the endpoints of $E_2$ by a path
  $\alpha_2\sub \partial Z_2$ of $k$-edges that avoids $p$,
  etc. In this way, we obtain paths $\alpha_i\sub \partial Z_i$
  for $i=1, \dots, N$ whose concatenation gives a path $\alpha$
  of $k$-edges in $G'$ that joins $a$ and $b$ and does not
  contain $p$.  Then $\alpha$ consists of $k$-cells that do not
  contain $p$, and so $\alpha \cap W^k(p)=\emptyset$ (see Lemma~\ref{lem:flowerprop}~\ref{item:flower_prop3}). It follows
  that the compact sets $\alpha$ and $|K|\sub W^k(p)$ are
  disjoint, and so these sets have positive distance with
  respect to some base metric on $S^2$.
  
  By a small modification we can push $\alpha$ inside $\Om$ to
  obtain a path $\beta$ in $\Om$ that is disjoint from $|K|$ and
  joins $A'$ and $B'$. To do this, we slightly move the initial
  point $a$ of $\alpha_1$ into $\inte(X)\sub \inte(Z_1)\sub \Om$
  and the other endpoint of $\alpha_1$ to $\inte(E_1)\sub
  \Om$. We can join these new endpoints by a path $\beta_1$ that
  follows the original path $\alpha_1$ closely and stays inside
  $\inte(Z_1)\cup \inte(E_1)\sub \Om$. We slightly move the last
  point of $\alpha_2$ into $\inte(E_2)$, and join the endpoint
  of $\beta_1$ to this point by a path $\beta_2$ that follows
  $\alpha_2$ closely and stays inside $\inte(E_1)\cup
  \inte(Z_2)\cup \inte(E_2)\sub \Om$. Continuing in this way, we
  get a collection of paths $\beta_1, \dots, \beta_N$ whose
  concatenation is a path $\beta$ in $\Om$ that joins $A'$ and
  $B'$ and stays so close to the original path $\alpha$ that
  $\beta\cap|K|=\emptyset$.
 
  This contradicts the fact that $K'\sub |K| $ separates $A'$
  and $B'$ in $\Om$.  So $|K|$ is not contained in a $k$-flower.

  Inequality \eqref{eqCEL:KDn} now easily follows. Indeed,
  $L\coloneqq f^k(|K|)$ is a connected union of $(n-k)$-tiles. This set
  joins opposite sides of $\CC$. For otherwise, $L$ is contained
  in a $0$-flower (Lemma~\ref{lem:floweropp}) which in turn
  implies by Lemma~\ref{lem:mapflowers}~\ref{item:mapflowers3}
  that the connected set $|K|\sub f^{-k}(L)$ is contained in a
  $k$-flower; but we have just seen that this is not the
  case.
  
   Since $L$ joins opposite sides of $\CC$, the number of
  $(n-k)$-tiles in this set must be at least $D_{n-k}(f,\CC)$.
 Since  the set $K$ contains at least as many $n$-tiles as $L$
  contains $(n-k)$-tiles,  inequality \eqref{eqCEL:KDn} follows.
\end{proof}

\section{Short $e$-chains}
\label{sec:short-e-chains}
We can now prove 
Proposition~\ref{prop:macgrdn} and complete the proof 
 of Lemma~\ref{lem:Nn_degf} (and hence of  
 Theorem~\ref{thm:Qianthm}). In this section we  will  denote the length of an $e$-chain or a tile chain $P$ by $\#P$ instead of $\length(P)$.

\begin{proof}[Proof of Proposition~\ref{prop:macgrdn}]
 Since $f$ is an expanding Thurston map, we have $\#\post(f)\ge 3$
 (see Lemma~\ref{lem:no<3}).  We first assume  that $\CC$ is $f$-in\-vari\-ant. 
  
  Now we apply Lemma~\ref{lem:size_separate} in the
  following setting. Let $k=0$ and our $e$-chain of $0$-tiles
  consist of $Z_1=X^0_\mathtt {w}$ and $Z_2=X^0_\mathtt {b}$,
  i.e., the two $0$-tiles attached along  some $0$-edge $E_1$ 
  (which is necessarily on the boundary of both $Z_1$ and $Z_2$).  Then $\Omega$ is equal to $S^2$ with the union of the $0$-edges
  distinct from $E_1$ removed.  We can pick $0$-edges $E_0\sub
  Z_1$ and $E_2\sub Z_2$, so that $E_0$, $E_1$, $E_2$ are
  distinct. Moreover, if $\#\post(f)\ge 4$, we may assume that
  $E_0\cap E_2= \emptyset.$
 
Let $n\in \N_0$ and let $G$ be the graph as defined
  before Lemma~\ref{lem:G_Om}. In the present situation,
  the vertex set $V$  is equal to the set of all
  $n$-tiles. So $\#V=2 \deg(f)^n$. Let $A\sub V$ be the set of
  all $n$-tiles that are contained in $Z_1$ and meet $E_0$, and
  $B\sub V$ be the set of all $n$-tiles that are contained in
  $Z_2$ and meet $E_2$.  Both $|A|$ and $|B|$ contain two
  distinct $0$-vertices, namely the endpoints of $E_0$ and $E_2$,
  respectively. Moreover, $A'=|A|\cap \Om$ and $B'=|B|\cap \Om$
  are connected as follows from
  Lemma~\ref{lem:G_Om}~\ref{item:grsep1b} applied to the
  path $\ga$ given by parametrizations of the arcs $E_0$ and
  $E_2$, respectively.
 
  If $K\sub V$ separates $A$ and $B$ in $G$, then $\#K\ge
  D_n(f,\CC)$ by Lemma~\ref{lem:size_separate}. It 
  follows from  
  Theorem~\ref{thm:Menger} (see the discussion after this theorem) that there exists an  $A$-$B$-path in
  ${G}$ whose  length is bounded above by  $ \#V/
  D_n(f,\CC)$. This path gives an $e$-chain $P$ consisting of
  $n$-tiles $X_1, \dots, X_M$ with $X_1\in A$, $X_M\in B$, and
  \begin{equation*}
    \#P
    =M
    \le 
    \#V/ D_n(f,\CC)
    =2\deg(f)^n/ D_n(f,\CC). 
  \end{equation*}
  
  Then $X_1$ meets $E_0$ and $X_M$ meets $E_2$. If
  $\#\post(f)\ge 4$, then $E_0$ and $E_2$ are disjoint $0$-edges
  which implies that $|P|$ joins opposite sides of $\CC$.  Hence
  $M\ge D_n(f,\CC)$, and so
  $$ D_n(f,\CC)\le M\le  2\deg(f)^n/ D_n(f,\CC), $$
  which gives 
  \begin{equation} 
    \label{eq:dnuppbd}  
    D_n(f,\CC)\le \sqrt 2 \deg(f)^{n/2}. 
  \end{equation}  
  This is an upper bound for $D_n(f,\CC)$ as desired. 
 
 If   $\#\post(f)=3$, the argument  is along
  similar lines, but  slightly more subtle. In this case,
  $E_0,E_1,E_2$ are the three $0$-edges. The underlying set $|P|$
  of our $e$-chain $P$ meets $E_0$ and $E_2$. We want to show
  that it also meets $E_1$. To see this, we choose a path
  $\gamma$ that starts in an interior point of $X_1$, 
ends in an interior point of $X_M$, 
and stays inside $|P|\cap \Om$. This is
  possible, since $P$ forms an $e$-chain where successive tiles
  have a common $n$-edge whose interior is contained in
  $\Om$.

  Now $E_1$ splits our simply connected region $\Om$ into the
  complementary components $\inte(Z_1)$ and $\inte(Z_2)$. Since
  $\ga$ stays in $\Omega$, starts in $\inte(Z_1)$, and ends in
  $\inte(Z_2)$, it must meet $E_1$. Hence $|P|\supset \ga$ also
  meets $E_1$.  So $|P|$ meets all three $0$-edges, which means
  that this set joins opposite sides of $\CC$ (see
  Definition~\ref{def:connectop}). 
Again we
  have $M\ge D_n(f,\CC)$ and derive \eqref{eq:dnuppbd}. This
  completes the proof of the statement when $\CC$ is invariant.

  We now consider  the general case  when $\CC$ is not
  necessarily $f$-invariant. Then we can find an iterate $F=f^N$
  of $f$ that has an $F$-invariant Jordan curve $\CC'\sub S^2$
  with $\post(F)=\post(f) \subset\CC'$ (see
  Theorem~\ref{thm:main}).  Then by the first part of the proof, 
  \begin{equation*}
    D_k(F, \CC')\lesssim \deg(F)^{k/2}=\deg(f)^{kN/2}. 
  \end{equation*}
  Since the $k$-tiles for $(F,\CC')$ are the $(kN)$-tiles for
  $(f,\CC')$ (see Proposition~\ref{prop:celldecomp}~\ref{item:celdecompiter}), this means
  \begin{equation*}
    D_{kN}(f,\CC') \lesssim  \deg(f)^{kN/2}.   
  \end{equation*}
  If $n\in \N_0$ is arbitrary, we can write it as $n=kN+l$,
  where $k\in \N_0$ and $l\in \{0, \dots, N-1\}$.  Now we know
  by \eqref{DDk} in Lemma~\ref{lem:Dnprops} that
  \begin{equation*}
    D_{m+1}(f, \CC') \lesssim D_m(f,\CC') 
  \end{equation*}
  for $m\in \N_0$. 
  Applying this inequality at most $(N-1)$-times, we are   led
  to  
  \begin{align*}
    D_n(f, \CC')&=D_{kN+l}(f, \CC') \lesssim  D_{kN}(f,\CC') 
    \\  
    &\lesssim \deg(f)^{kN/2} \le  \deg(f)^{n/2}.
\end{align*} 
Moreover,  by \eqref{DDtilde} in Lemma~\ref{lem:Dnprops} we
know that
$$ D_n(f, \CC)\asymp D_{n}(f, \CC').$$ 
So 
$$D_n(f,\CC)\asymp  D_n(f,\CC') \lesssim \deg(f)^{n/2}. $$
Since in all these inequalities the implicit multiplicative
constants are independent of  $n$ (or of $k$ and $m$  as in some of
the previous inequalities), the statement follows.
\end{proof} 

In the following two lemmas we make the assumption that
$f\: S^2\ra S^2$ is an expanding Thurston map with an
$f$-invariant Jordan curve $\CC\sub S^2$ with $\post(f)\sub \CC$
that satisfies condition \eqref{eq:maxgrDn} in
Theorem~\ref{thm:Qianthm}. We allow the possibility that $f$ has
periodic critical points. Tiles in these statements will be for
$(f, \CC)$.
 
\begin{lemma} 
  \label{lem:lpd1}  
  Let $n,k\in \N$, $n\ge k$, and suppose that $X^k$ and $Y^k$
  are two $k$-tiles that are both contained in a  $(k-1)$-tile
  $U^{k-1}$.
 
  Then there exists an $e$-chain $P$ consisting of $n$-tiles
  that starts in an $n$-tile contained in $X^k$, ends in an
  $n$-tile contained in $Y^k$, and satisfies
  \begin{equation} 
    \label{eqCEL:lpd1} 
    \#P \le c' \deg(f)^{(n-k)/2},
 \end{equation} 
 where  $c'>0 $ only depends on $f$ and $\CC$.  
\end{lemma} 
 
The lemma essentially says that under our assumptions, $k$-tiles
with a common parent can be joined by an $e$-chain of $n$-tiles
with controlled length.

\begin{proof}   
  Since $f^{k-1}|U^{k-1}$ is a homeomorphism that maps $k$-tiles
  to $1$-tiles, the number of $k$-tiles contained in $U^{k-1}$
  is bounded by $N_0\coloneqq 2\deg(f)$, the number of $1$-tiles.
 
  We can pick a path $\ga\sub \inte(U^{k-1})$ that joins $X^k$
  and $Y^k$. Then $\ga$ only meets $k$-tiles contained in
  $U^{k-1}$. By Lemma~\ref{lem:G_Om}~\ref{item:grsep1a} this
  implies that there exists a simple $e$-chain consisting of
  $k$-tiles
  \begin{equation*}
    Z_1=X^k, Z_2, \dots, Z_N= Y^k 
  \end{equation*}
 with $N\le N_0$.
 
 Let $n\geq k$ be arbitrary, and  $G$ be the graph as defined before Lemma~\ref{lem:G_Om} 
 for $Z_1,\dots, Z_N$. Its vertex set  $V$
 is given by the set of  all $n$-tiles contained in any of the tiles $Z_i$.
  Let
 $A\sub V$ and $B\sub V$ 
 consist of all $n$-tiles contained in $Z_1=X^k$ and $Z_N=Y^k$,
 respectively. Then $|A|=X^k$ and $|B|=Y^k$ contain the
 $k$-vertices on the boundary of $X^k$ and $Y^k$, respectively,
 and hence both contain at least two distinct $k$-vertices.
 Moreover, if we define $\Om$ as in \eqref{defOmega}, and
 $A'=|A|\cap \Om$ and $B'=|B|\cap \Om$, then
 $\inte(X^k)\sub A'\sub X^k$ and $\inte(Y^k)\sub B'\sub Y^k$.
 This implies that $A'$ and $B'$ are connected.

  So Lemma~\ref{lem:size_separate} applies and we conclude from  
 (Menger's) Theorem~\ref{thm:Menger} that there are at least 
 $D_{n-k}(f, \CC)$  disjoint and simple  $A$-$B$-paths in $G$.  Let $P$ be a simple  $A$-$B$-path of minimal
  length. Note that $P$ forms an $e$-chain of $n$-tiles whose
  first $n$-tile is contained in $X^k$ and its last in
  $Y^k$. The number $\#V$ of vertices in $G$ is equal to the
  number of $n$-tiles contained in any of the $k$-tiles $Z_i$, and
  hence bounded by $2 N_0\deg(f)^{n-k}$. It
  follows that
  \begin{equation*}
    \#P \cdot D_{n-k}(f, \CC)\le  \#V\le 2N_0 \deg(f)^{n-k}. 
  \end{equation*}
  Using the lower bound \eqref{eq:maxgrDn} we conclude
  \begin{equation*}
    \# P \le c' \deg(f)^{(n-k)/2},
  \end{equation*}
  where $c'>0 $ only depends on $f$ and $\CC$. 
\end{proof}

\begin{lemma} 
  \label{lem:lpd2}
  Let $n,k\in \N_0$, $n\ge k$, $X^n$ and $Y^n$ be two $n$-tiles
 that are both contained in a  $k$-tile
  $U^{k}$.   Then there exists an $e$-chain $P$ of
  $n$-tiles that starts in $X^n$, ends in $Y^n$, and satisfies
  \begin{equation}
    \label{eqCEL:lpd2} 
    \#P \le c' \sum_{i=0}^{n-k} 2^{n-k-i} \deg(f)^{i/2},
  \end{equation}
  where  $c'\ge 1$ is a constant only depending on $f$ and
  $\CC$.  
\end{lemma} 
 
As we will see in the proof, we can take the same constant $c'$
in \eqref{eqCEL:lpd2} as in \eqref{eqCEL:lpd1} if $c'\ge 1$ as
we may assume.  
%
%

\begin{figure}
  \centering
  \begin{overpic}
    [width=10cm, tics=10,
    ]
    {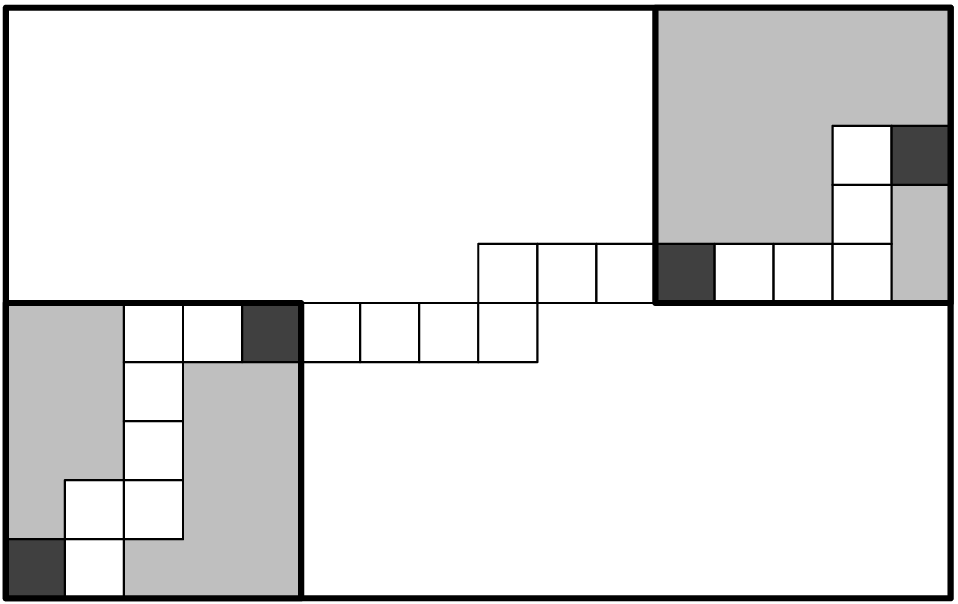}
    \put(-7,15){$X^k$}
    \put(-7,1){$X^n$}
    \put(22,13){$P_0$}
    \put(26,33){$\widetilde{X}^n$}
    \put(43,34){$P_1$}
    \put(69,25.5){$\widetilde{Y}^n$}
    \put(79,44){$P_2$}
    \put(101,46){$Y^n$}
    \put(62,47){$Y^k$}
    %
    \put(55,5){$U^{k-1}$}
  \end{overpic}
  \caption{Connecting tiles by short $e$-chains.}
  \label{fig:tiles_short_echains}
\end{figure}
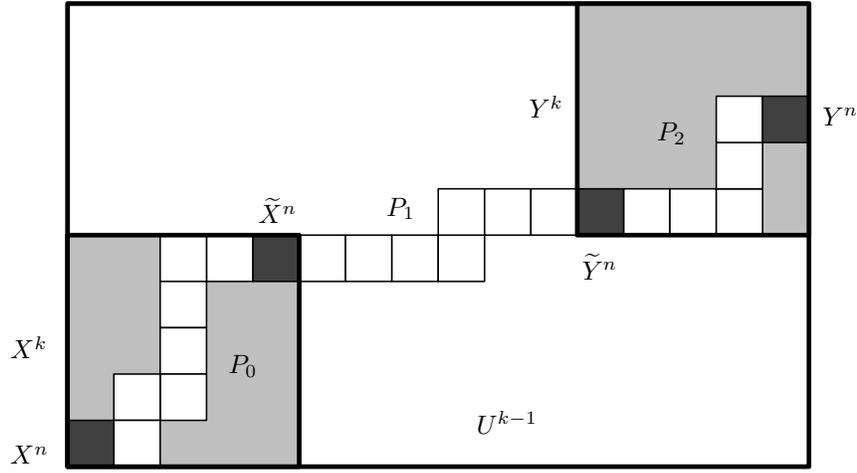

\begin{proof} 
  We prove this for fixed $n\in \N$ (and arbitrary tiles $X^n$
  and $Y^n$) by downward induction on $k=n, n-1, \dots, 0$.

  The statement is true for $k=n$. Indeed, in this case
  $X^n=Y^n=U^k$, and so the single tile $X^n=Y^n$ forms a
  suitable chain $P$.  We have $\#P\le 1$ which gives a bound as
  in \eqref{eqCEL:lpd2} for $n=k$ if $c'\ge 1$.

  Now we assume that the statement is true for some number
  $0< k\le n$. We need to show that it is also true for $k-1$.
  To see this, suppose that we have $n$-tiles $X^n$ and $Y^n$,
  and a $(k-1)$-tile $U^{k-1}$ with $X^n, Y^{n}\sub U^{k-1}$.
  Then there exist unique $k$-tiles $X^k$ and $Y^k$ with
  $X^n\sub X^k\sub U^{k-1}$ and $Y^n\sub Y^k\sub U^{k-1}$.  
  See Figure~\ref{fig:tiles_short_echains} for an illustration.

  By
  Lemma~\ref{lem:lpd1} there exists an $e$-chain $P_1$ of
  $n$-tiles with
  \begin{equation*}
    \#P_1 \le c' \deg(f)^{(n-k)/2}
  \end{equation*}
  that starts in an $n$-tile $\widetilde X^n\sub X^k$ and ends
  in an $n$-tile $\widetilde Y^n\sub Y^k$.  We can now apply the
  induction hypothesis to the $n$-tiles
  $X^n, \widetilde X^n\sub X^k$ to find an $e$-chain $P_0$ of
  $n$-tiles that starts in $X^n$, ends in $\widetilde X^n$, and
  satisfies
  \begin{equation*}
    \# P_0\le c' \sum_{i=0}^{n-k} 2^{n-k-i} \deg(f)^{i/2}.
  \end{equation*}
  Similarly, we can find an $e$-chain $P_2$ of $n$-tiles that
  starts in $\widetilde Y^n$, ends in $Y^n$, and satisfies
  \begin{equation*}
    \#P_2\le c' \sum_{i=0}^{n-k} 2^{n-k-i} \deg(f)^{i/2}.
  \end{equation*}
  Concatenating $P_0$, $P_1$, $P_2$, leads to an $e$-chain $P$
  of $n$-tiles that starts in $X^n$, ends in $Y^n$, and
  satisfies
  \begin{align*}
    \#P &\le  \#P_0 +\#P_1 +\# P_2 \\
        &\le 2 c'  \sum_{i=0}^{n-k} 2^{n-k-i} \deg(f)^{i/2} + c' \deg(f)^{(n-k)/2}\\
        &\le  c'\sum_{i=0}^{n-k+1} 2^{n-k+1-i} \deg(f)^{i/2}. 
  \end{align*} 
  The statement follows. 
\end{proof} 

If we assume $\deg(f)^{1/2}>2$, the estimate  \eqref{eqCEL:lpd2} can be simplified 
and gives an upper bound as needed for the proof of Lemma~\ref{lem:Nn_degf}.
This is the reason why this assumption was made in the lemma.

\begin{cor}
  \label{cor:length_echains} Suppose in  Lemma~\ref{lem:lpd2} we 
 make the additional assumption  that
  $\deg(f)^{1/2} >2$. Then the $e$-chain $P$ in 
  \eqref{eqCEL:lpd2} 
  satisfies 
  \begin{equation*}
    \# P \leq C \deg(f)^{(n-k)/2}, 
  \end{equation*}
  where $C$ is a constant depending only on $\CC$ and $f$.
\end{cor}

\begin{proof}
If $\deg(f)^{1/2} >2$,   then  the right hand side in  \eqref{eqCEL:lpd2} is a 
   geometric sum with terms that increase with $i$; so  up to a multiplicative constant this sum  is dominated by its last 
  term. 
\end{proof}

We are now ready  to prove
Lemma~\ref{lem:Nn_degf}. 

\begin{proof}[Proof of Lemma~\ref{lem:Nn_degf}.]
  Let $f\colon S^2\to S^2$ be an expanding Thurston map with
  $\Lambda\coloneqq  \deg(f)^{1/2}>2$. Suppose $f$  has an invariant Jordan curve $\CC\subset
  S^2$ with $\post(f) \subset \CC$ that satisfies
  \eqref{eq:maxgrDn}. 
  
  Let $x,y\in S^2$ be distinct. In Lemma~\ref{lem:Nn_lower_bd}
  we already saw that $N_n(x,y) \gtrsim \Lambda^{n-m(x,y)}$ for
   $n\geq m(x,y)+1$. This is the
  desired lower bound. 

   To obtain the upper bound, set $k\coloneqq m(x,y)\in\N_0$. Then
  by definition of $k=m(x,y)$ there exist $k$-tiles $X^k$ and
  $Y^k$ with $x\in X^k$, $y\in Y^k$, and
  $X^k\cap Y^k\ne \emptyset$. For each $n\ge k$ we can pick
  $n$-tiles $X^n, \widetilde X^n, \widetilde Y^n, Y^n$ with
  $X^n,\widetilde X^n\sub X^k$, and
  $Y^n,\widetilde Y^n\sub Y^k$, as well as
  \begin{equation*}
    x\in X^n, y\in Y^n,   
    \text{ and }
    \widetilde X^n \cap \widetilde Y^n \ne \emptyset. 
  \end{equation*}
  Then by
  Corollary~\ref{cor:length_echains} there
  exists an $e$-chain $P_1$ of $n$-tiles with
  $\# P_1\lesssim \Lambda^{n-k}$ whose first tile is $X^n$ and
  whose last tile is $\widetilde X^n$. Similarly, there exists
  an $e$-chain $P_2$ of $n$-tiles with
  $\#P_2\lesssim \Lambda^{n-k}$ whose first tile is
  $\widetilde Y^n$ and whose last tile is $Y^n$. Since
  $\widetilde X^n \cap \widetilde Y^n \ne \emptyset$, the union
  $P=P_1\cup P_2$ is a chain of $n$-tiles (not necessarily an
  $e$-chain) with $\#P\lesssim \Lambda^{n-k}$ whose first tile
  is $X^n$ and whose last tile is $Y^n$.
 
  This implies that
  $$ N_n(x,y) \le \#P \lesssim \Lambda^{n-k}= \Lambda^{n-m(x,y)}$$
  for $n\ge k$. This is the required upper bound, finishing the
  proof.
\end{proof}

By  establishing  Lemma~\ref{lem:Nn_degf} we have also completed the proof of 
Theorem~\ref{thm:Qianthm} (see the end of
Section~\ref{sec:visual-metrics-that}).

\ifthenelse{\boolean{singlechapter}}{


%


\chapter{Outlook and open problems}
\label{ch:outlook}

In this final chapter we give an outlook on  further results that are 
related to the major themes in  this book. We do not try to be
exhaustive, but rather intend  the discussion   as a first entry
point into  various other investigations. We combine this with a presentation of some
open problems that will hopefully stimulate future  research.

\subsection*{Markov partitions for Thurston maps}
The basis of our combinatorial approach to the study of an
expanding Thurston map $f$ is to consider the cell
decompositions induced by a Jordan curve containing the
postcritical points of $f$ as in Chapter~\ref{cha:celldecomp}. If  this Jordan curve is $f$-invariant,  we obtain a Markov
partition and can  describe our map in
a combinatorial fashion by  a two-tile subdivision rule (see 
Section~\ref{sec:subdivisions}). Recall that
Theorem~\ref{thm:main} ensures the existence of  an
$f^n$-invariant   curve for each sufficiently large $n\in
\N$.   

We know (see Example~\ref{ex:noinvCC})   that an expanding Thurston
 map $f$ itself   does not necessarily have an $f$-invariant 
 curve 
and that in general one has to    
pass to a suitable  iterate  in
order to obtain such a  curve. 
This naturally leads to the question whether one can bound the
order of this iterate in terms of some natural invariants  of
the map.

\begin{prob}
  \label{prob:estimate_n} 
  Let $f\: S^2\ra S^2$ be an expanding Thurston map. Is there a
  number $N_0\in \N$, depending on some natural data such as
  $\deg(f)$, $\#\post(f)$, and $\Lambda_0(f)$ such that for all
  $n\geq N_0$ there exists an $f^n$-invariant Jordan curve
  $\CC\sub S^2$ with $\post(f)\subset \CC$?
\end{prob}
Recall that $\Lambda_0(f)$  denotes the combinatorial expansion factor of $f$ (see Chapter~\ref{cha:combexpfac}).

If $f\: S^2\ra S^2$  is an  expanding Thurston map 
 and $\CC\subset S^2$ an $f^n$-invariant   Jordan curve with
$\post(f) \subset \CC$, then the  set $$\CC_n\coloneqq \CC \cup f^{-1}(\CC) \cup \dots
\cup f^{-n + 1}(\CC)$$ is easily seen to be $f$-invariant. 
This gives some type of  Markov partition for $f$, but it will not be cellular 
  as defined  
in Section~\ref{s:celldecomp}. In particular, $S^2\setminus
\CC_n$ may have  infinitely many components and in general   we have very little control over the 
geometric shapes of the ``tiles''  in this   partition.

\begin{prob}
  \label{prob:Markov_part}   
  Does every expanding Thurston map $f\:S^2\ra S^2$ have a
  (finite) cellular Markov partition?
\end{prob}

We conjecture that the answer should be affirmative. To prove
this statement, one essentially has to construct a finite
connected graph $G\subset S^2$ with $\post(f)\subset G$ that is
$f$-invariant.  A different way to phrase the problem is to ask
whether $f$ (and not some iterate $f^n$) can be described by a
(suitably defined) \emph{$k$-tile subdivision rule}. Here $k$
would be the number of components of $S^2\setminus G$.

An interesting special case is when the invariant graph
 is  a tree $T$ and so $S^2\setminus  T$ is connected (then we have a $1$-tile subdivision rule). 

 Cannon-Floyd-Parry \cite[Theorem 3.1]{CFP10} showed  that such an invariant tree exists for each 
sufficiently high iterate of a Latt\`{e}s map with signature
$(2,2,2,2)$.  They also observed 
\cite[Section~4]{CFP10} that it is
indeed necessary to take an iterate here, because they found examples of such Latt\`{e}s maps $f$ for which no
$f$-invariant tree $T$ with $\post(f)\subset T$ exists.

Farrell and Jones \cite{FJ79} constructed finite cellular Markov partitions  in a more general context, but they also had to pass to 
sufficiently high  iterates of the maps considered to guarantee existence of the Markov partition (their definition of this concept differs from ours; for yet another definition of a Markov partition,  see
 \cite[Definition~4.5.1] {PU}).

Recall from Chapter~\ref{cha:symdym} that each expanding
Thurston map $f$ can be described as a factor of the
left shift on the space of infinite words in an alphabet of  $d=\deg(f)$ elements. This 
can be used to obtain a Markov partition for $f$, but again we will have very little geometric control for the geometry of corresponding ``tiles''.

Markov partitions for Latt\`{e}s maps and how they behave under certain perturbations were discussed in \cite{Be94}.
Related results can also be found  in \cite{Re15}.

Another way to find Markov partitions  is  via the iterated monodromy
group  (the concept of a \emph{limit space} is relevant here; see \cite[Chapter~3]{Ne}) or 
by using invariant Peano
  curves (see the discussion below).
 These methods give  ``tiles'' with a complicated 
 geometric structure.

 In this work we mostly considered Thurston maps that
  are expanding. One may ask whether combinatorial descriptions
  as for these maps exist for more general types of maps.

\begin{prob}
  \label{prob:non-expanding}
  \index{Thurston map!iterate of}
  \index{iterate of Thurston map}
  \index{F f@$F=f^n$}
  Let $f\:S^2 \ra S^2 $ be a Thurston map (not necessarily
  expanding). Is there a Jordan curve $\CC\sub S^2$ with
  $\post(f)\subset \CC$ that is invariant for some iterate
  $f^n$? Are there other natural partitions of the sphere
  $S^2$ that are invariant (in a suitable sense) under   $f$ or
  some iterate $f^n$? 
\end{prob}

Related to this is a variant of Theorem~\ref{thm:main}  proved in
\cite{GHMZ}. Namely, let $f\colon \CDach \to \CDach$ be a
 rational Thurston map  whose  Julia set is a
Sierpi\'{n}ski carpet. Then for each sufficiently large $n\in \N$
there is a Jordan curve $\CC\subset \CDach$ with $\post(f)
\sub  \CC$ that is invariant for  $f^n$.

When $f$ is a polynomial, natural partitions of $\CDach$ can be  obtained from 
\emph{external rays} (see \cite{DH84}) and the related  
\emph{Yoccoz puzzles} (see \cite{Hu} and \cite{Mi00}). A lack
of similar combinatorial methods  for general  rational maps is one of the main
reasons why their study is harder than the study of polynomials.

Another question is whether and how our results extend to maps
that are not necessarily postcritically-finite.
The theory of {\em coarse expanding
  dynamical systems} developed in \cite{HP} should be relevant here. 

\begin{prob}
  \label{prob:invC_JS2}
  Let $f\colon \CDach\to \CDach$ be a rational map (not necessarily
  post\-crit\-i\-cally-finite) whose  Julia set is the
  whole Riemann sphere $\CDach$. Does there exist a natural
  combinatorial description of $f$ or some iterate $f^n$?
\end{prob}
The rational maps $f\: \CDach \ra \CDach$ of a given degree
$d\ge 2$ form a complex manifold $\mathcal{R}_d$ of dimension
$2d+1$.  Rees showed  \cite{Re} that the set of points in
$\mathcal{R}_d$ where the corresponding rational map $f$ has a
Julia set equal to the whole sphere has positive measure with
respect to the natural measure class on $\mathcal{R}_d$. Such points in $\mathcal{R}_d$ and the corresponding maps   can be  obtained by a slight perturbation of 
certain expanding Thurston maps. It would
be interesting to find combinatorial descriptions of expanding
Thurston maps that change under such perturbations in a
controlled manner.

Instead of asking whether good combinatorial models exist for
(expanding) Thurston maps, one may ask for good analytical
models. 
As we have seen, the dynamics of an expanding  Thurston map $f$ generates a 
class of visual metrics, and so a fractal geometry on the underlying $2$-sphere. This does not rule out that 
 the map $f$ can actually be described by a smooth model.   Li  showed  \cite{Li3} that no expanding Thurston map with  periodic critical points is conjugate to a smooth map, but the following problem is  still open. 

 \begin{prob}
   \label{prob3}  
   Is every  expanding Thurston map without   periodic 
   critical points  topologically conjugate to a smooth 
  expanding Thurston map on $\CDach$?       
\end{prob}
This question was raised by K.~Pilgrim.  We expect  this to be
true  for at least every  sufficiently high iterate of the
given map.

If we assume 
that  a  Thurston map is given by a two-tile subdivision rule or in some other combinatorial way, then 
one wants to know which   information about
the map can be extracted from this combinatorial
description.  This general question is a major theme in the study of the dynamics
of polynomials pioneered by Douady and Hubbard (see \cite{DH84}).


We know that  every two-tile subdivision rule is realized by a
Thurston map that is  unique up to Thurston equivalence. In contrast, a Thurston map may be realized by  combinatorially different two-tile subdivision rules.  

\begin{prob}
  \label{prob:diff_sub_same_f} Suppose two 
  Thurston
  maps  $f$ and $g$ realize different two-tile subdivision
  rules. How can one decide from combinatorial data  whether the
  maps  are 
  Thurston equivalent?
\end{prob}

Of course, there are several simple 
necessary conditions such as 
$\deg(f)=\deg (g)$ and $\#\post(f)= \#\post(g)$,  whose
validity can easily be verified  from the subdivision rules. In
addition, the maps  must have the same ramification
  portrait  (see Section~\ref{sec:defin-thurst-maps}). A related
question, namely when polynomials with the same ramification 
portrait are Thurston equivalent, was answered in \cite{BN}. The
major tool used there was the  iterated monodromy group as
discussed below. Closely related is the biset associated with a
Thurston map, which may also be used to study Thurston
equivalence (see in particular \cite{BD}). The questions of whether two  Thurston maps are equivalent or whether a Thurston map is equivalent to a
 rational map are  decidable (see
\cite{BBY12}).

\begin{prob}
  \label{prob:comb_exp}
  Let $f$ be an expanding Thurston map that  realizes a two-tile
  subdivision rule. Is there an effective way to compute the
  combinatorial expansion factor $\Lambda_0(f)$  from the combinatorial description? 
\end{prob}

Since   $\Lambda_0(f)$ is defined 
as a limit,  a priori  one cannot expect to find $\Lambda_0(f)$ by a finite 
procedure. However, if $f$ is a Latt\`es or Latt\`es-type map, then $\Lambda_0(f)$ is the smallest 
 absolute value of the two   eigenvalues   of the  matrix describing the underlying 
 torus endomorphism \cite{Qian15}. In general, one  may  speculate that if $f$ realizes a two-tile subdivision rule with underlying cell decompositions $\DD^1$ and $\DD^0$, then 
  $\Lambda_0(f)$ is related to  the spectral radius of a matrix  that is obtained from the incidence relations of the cells in $\DD^1$ and their images under $f$ in $\DD^0$.

\subsection*{Conformal dimension of the visual sphere}
Recall from Theorem~\ref{thm:S2vsf}~\ref{item:S2qsphere} 
 that if an expanding Thurston
map $f$ is topologically conjugate to a rational map, then 
its  visual sphere is a quasisphere. In other words, there  is a
quasisymmetry  $(S^2,\varrho) \to (\CDach,\sigma)$, where
$\varrho$ is a visual metric for $f$ and $\sigma$ is the chordal
metric on $\CDach$. In particular, $(S^2,\varrho)$ is then 
 quasisymmetrically equivalent to an  Ahlfors $2$-regular space, namely 
 the Riemann sphere  $(\CDach,\sigma)$ 
 
A closely related  question is how much the
Hausdorff dimension of a metric space $(X,d)$ can be lowered by a
quasisymmetric map. 
%
%
This can be  measured by  the infimum of the Hausdorff
dimensions of all metric spaces 
$(X', d')$ that are  quasisymmetrically
equivalent to $(X,d)$. Actually, often a more relevant quantity 
is   the \emph{(Ahlfors regular) conformal
  dimension} of $(X,d)$, where one takes the corresponding infimum only over 
   Ahlfors regular metric spaces $(X', d')$ (see \cite{MT10}). For metric spheres $(S^2,d)$ that are
not quasispheres  the conformal dimension measures   in a sense  by how much
$(S^2,d)$ fails to be a quasisphere. 

If $f\:S^2\ra S^2$ is an expanding Thurston map without periodic critical points, then $S^2$ equipped   with a visual metric $\varrho$ is Ahlfors regular 
  (Proposition~\ref{prop:Ahlforsreg}). Since all visual metrics for $f$ are quasisymmetrically equivalent,  the  conformal dimension of
the visual sphere  $(S^2, \varrho)$ only depends on $f$ and not on the choice of the visual metric $\varrho$. 
This is an important numerical invariant of the fractal geometry of $(S^2,\varrho)$.

\begin{prob} 
  \label{prob4} 
  Is it possible to determine the (Ahlfors regular) conformal dimension  of the visual 
  sphere of  an expanding
  Thurston map in terms of its dynamical data?
\end{prob}
  
Bonk-Geyer-Pilgrim \cite{Bo} stated a conjecture expressing this
conformal dimension in terms of eigenvalues of certain matrices
related to  the dynamics of the map.  One of the inequalities relating the quantities 
 in  this conjecture was
 established  by Ha\"\i ssinsky-Pilgrim \cite{HP08}.    

In general, the infimum defining the  conformal dimension of a metric space may not be attained as a minimum. 
For metric spheres that arise as boundaries at infinity of Gromov
hyperbolic groups the following related result was proved in
\cite[Theorem~1.1]{BK05}.
 
\begin{theorem}
  \label{thm:ahl_reg_conf_dim_Groups}
  Let $G$ be a Gromov hyperbolic group whose  boundary at infinity
  $\partial_\infty G$ is homeomorphic to a $2$-sphere. If the  conformal dimension $Q$ of $\partial_\infty G$ (equipped
  with a visual metric) is attained as a minimum, then $Q=2$ and $\partial_\infty G$ is a quasisphere. \end{theorem}
  
  As we discussed in Section~\ref{sec:Cannconj},   then there exists   an action of $G$ on hyperbolic 
  $3$-space $\Halb^3$ that is geometric, i.e.,  isometric, properly discontinuous, and   cocompact. 
A corresponding result for expanding Thurston
maps was established  in \cite{HP14}.

\begin{theorem}
  \label{thm:conf_dim_attained_Thurs}
  Let $f\colon S^2\to S^2$ be an expanding Thurston map without
  periodic critical points and suppose the conformal dimension
  $Q$ of its visual sphere is attained as a minimum. 
  
  Then either
  $Q=2$ and $f$ is topologically conjugate to a rational map,
  or $f$ 
  is a Latt\`{e}s-type map with signature $(2,2,2,2)$ and the linear part 
  $L_A$ of the affine map $A$ associated with $f$  has two distinct real eigenvalues
  $>1$ (in which case $Q>2$).
  \end{theorem}
  
   
 Ha\"{i}ssinsky and Pilgrim actually proved their result in greater generality 
 for 
\emph{topologically coarse expanding conformal maps}. These
 maps cannot have periodic critical points, but are not necessarily 
 post\-cri\-ti\-cally-finite.

\subsection*{Equivalence to a rational map}
\label{sec:equiv-rati-map}
As we already discussed in the introduction,
Theorem~\ref{thm:S2vsf}~\ref{item:S2qsphere} gives a criterion
for an expanding Thurston map to be topologically conjugate to a
rational map quite different from Thurston's theorem
(Theorem~\ref{thm:Thurston}). The latter theorem is proved by
methods fundamentally different from our combinatorial approach:
one considers a suitable Teichm\"uller space $\mathcal{T}$ and
studies a map $f^*\: \mathcal{T}\ra \mathcal{T} $ (the {\em
  Thurston pull-back map}) induced by the given Thurston map $f$
(with hyperbolic orbifold). Then $f$ is equivalent to a rational
map if and only if the induced map $f^*$ has a fixed point in
$\mathcal{T}$ \cite{DH}.  It is very desirable to reconcile these
points of view.

\begin{prob}
  \label{prob:comb_crit_f_rat}
 Is it possible  to give  a proof of Thurston's theorem  that   does not
  use Teich\-m\"{u}l\-ler theory and is  based on a 
  combinatorial description of the given map?
\end{prob}

As we know, such combinatorial descriptions are given, for example, 
by two-tile subdivision rules.

\begin{prob}
  \label{prob:f_rational}
Suppose an expanding Thurston map  $f$  realizes a two-tile
  subdivision rule. Is it possible to decide from the
  subdivision rule whether $f$ is Thurston equivalent to a
  rational map? 
\end{prob}

Of course, one should interpret this  as asking for a criterion 
that is easier to check in practice than the one provided by
Thurston's theorem.

An additional motivation for considering these problems  is that by 
Theorem~\ref{thm:S2vsf}~\ref{item:S2qsphere} they are closely related to the question of   characterizing quasispheres.
So their solution may give new ideas that could  also be used in the group setting and possibly  be applied for progress on  Cannon's
conjecture (see Section~\ref{sec:Cannconj}).


%
Thurston's theorem (Theorem~\ref{thm:Thurston}) is most useful as
a negative  criterion, 
 allowing  one to decide when a Thurston map is not equivalent
to a rational map by finding a Thurston obstruction.
For an obstructed Thurston map, i.e., a Thurston map that is
not equivalent to a rational map, in general there may 
be several different Thurston obstructions. Pilgrim showed that  in this case  one can single out  a ``canonical'' Thurston obstruction  \cite{Pi01}.

For a Thurston map for which  each cycle in the postcritical set
contains a critical point, a sufficient criterion for the map to
be equivalent to a rational map was established by Dylan Thurston
in joint work with Kahn and Pilgrim (see \cite{Th15} and
\cite{KPT}).

\subsection*{Special classes of maps}
\label{sec:nearly-eucl-maps}
 Since Latt\`{e}s maps form the best understood subclass of 
Thurston maps, it is natural to investigate    maps that are closely related to 
Latt\`{e}s maps. One such class, the nearly Euclidean 
  Thurston maps,  was  introduced and studied in \cite{CFPP12}. By definition 
a Thurston map $f\colon S^2\to S^2$ is called 
\emph{nearly Euclidean}\index{nearly Euclidean Thurston map}\index{Thurston map!nearly Euclidean} 
if it has exactly four 
postcritical points and every critical point has local degree
equal to $2$.  These maps can  
have hyperbolic orbifolds 
and  
periodic critical  points. 

 One can construct such maps as follows. 
Let 
 $g\colon S^2\to S^2$ be  a Latt\`{e}s-type map with signature 
$(2,2,2,2)$,  and 
$h\colon S^2\to S^2$ be    an orien\-tation-preserving homeomorphism such that 
 $h(\post(g)) \subset
g^{-1}(\post(g))$. 
Then $f= h\circ g$ is a nearly Euclidean
Thurston map. 

In general, it is difficult to use Thurston's theorem and  check whe\-ther a given Thurston map (with hyperbolic orbifold) is equivalent to a rational map, because 
the map may have infinitely many   invariant multicurves each of which could be a Thurston obstruction. So potentially one has  to verify infinitely many conditions. 
For nearly Euclidean Thurston maps, however,  it is possible to give an explicit algorithm that  decides whether the map is equivalent to a rational map. 

A 
\emph{Bely\u{\i} map}\index{Belyi map@Bely\u{\i} map}
is a holomorphic map $F\colon S\to
\CDach$, defined on a compact Riemann surface $S$, that is
ramified 
over three points. This means that $F$ has exactly three critical
values which  are usually taken to be $\{0,1,\infty\}$. In this case, the set 
$G\coloneqq F^{-1}([0,1])\subset S$ is a topological graph  with vertex
set $F^{-1}(\{0,1\})$. Its edges  are  the closures of the  components of
$F^{-1}((0,1))$.  The graph $G$ is bipartite if one distinguishes 
the vertices in $F^{-1}(0)$ and $F^{-1}(1)$. They are 
often marked by black and white dots, respectively. 
 The resulting diagram is called the
\emph{dessin d'enfant}\index{dessin d'enfant} 
of $F$ (introduced by Grothendieck in \cite{Gro97}). It determines $F$ up to pre- and
postcomposition with conformal maps. In particular, it defines $S$
up to conformal equivalence. Bely\u{\i}'s theorem says that each
non-singular algebraic curve defined over the field $\overline{\Q}$ of algebraic numbers
 can  be represented  by a Bely\u{\i} map. The \emph{absolute
  Galois group} $\operatorname{Gal}(\overline{\Q}/\Q)$ (i.e., the
group of field automorphisms of  $\overline{\Q}$ that fix
$\Q$) acts on these algebraic curves, and so on the set of
dessins d'enfants. See \cite{LanZvo} for an introduction to this
subject. 

Let  $f\colon \CDach \to \CDach$ be a Bely\u{\i} map
defined on $\CDach$ whose set of critical values is given by 
$\{0,1,\infty\}$. In general, $f$ is not a dynamical object, because the iterates of $f$ will not be  Bely\u{\i} maps in general. This is true if $\{0,1,\infty\} \subset f^{-1}(\{0,1,\infty\})$. Note that this can always be achieved by precomposing $f$ with a suitable M\"{o}bius
transformation. In this case, $f$ is a rational Thurston map with
$\post(f)= \{0,1,\infty\}$. The
representation of such a map $f$ by its dessin d'enfant is
closely related to  our description of $f$ by cell
decompositions. In fact, 
if we choose $\CC= \widehat{\R}$, then 
the $1$-skeleton $f^{-1}(\CC)=f^{-1}(\widehat{\R})\supseteq f^{-1}([0,1])$  of $\DD^1(f,\CC)$ contains the dessin
d'enfant of $f$.


The relation between Hubbard trees and dessins d'enfants for 
polynomials $P$ (with $\post(P) =\{0,1,\infty\}$ equal to 
the set of critical values of $P$) was investigated in \cite{Pi00}. 

\subsection*{Invariant Peano curves and mating of polynomials}
If we  denote by $\Sph^1=\partial \D$ the unit circle in $\C$ and
by $S^2$ a $2$-sphere (as before), then a {\em Peano curve} in $S^2$ is a continuous and surjective map
$\ga \: \Sph^1\ra S^2$. 
The following result was established  in
\cite{Me09b}.

\begin{theorem}
  \label{thm:inv_peano}
  \index{Thurston map!iterate of}
  \index{iterate of Thurston map}
  \index{F f@$F=f^n$}
  Let $f\: S^2\ra S^2$ be an expanding Thurston map. Then for each sufficiently
  high iterate $F=f^n$ there is a 
 {Peano curve} $\gamma\colon \Sph^1 \to S^2$  such that
  $F(\gamma(z))= \gamma(z^d)$ for all $z\in \Sph^1 $, where $d=\deg
  (F)$.
\end{theorem}
A  Peano curve $\ga$ as in this statement is called {\em $F$-invariant}. 
One can actually say more about $\ga$ 
here;
namely, if we identify $S^2$ and $\CDach$ so that $\Sph^1\sub \CDach \cong S^2$, then there 
exists a pseudo-isotopy  $H\: S^2\times I\ra S^2$ with 
 $H_0=\id_{S^2}$ and $H_1(z)=\ga(z)$ for 
 $z\in \Sph^1\sub S^2$.

Theorem~\ref{thm:inv_peano} says
 that 
  the following diagram commutes:
  \begin{equation*}
    \xymatrix{
      \Sph^1 \ar[r]^{z\mapsto z^d} \ar[d]_{\gamma}
      &
      \Sph^1 \ar[d]^{\gamma}
      \\
      S^2 \ar[r]_F & S^2\rlap{.}
    }
  \end{equation*}

On a more intuitive level, 
the theorem
can  be phrased as follows:  if we 
wrap $\Sph^1$  around itself $d$ times, then we obtain    the map $F$ through the  parametrization of $S^2$ by the Peano curve $\ga$.

The construction of the invariant Peano curve $\gamma$ in
Theorem~\ref{thm:inv_peano} is very similar  to the iterative
construction of  invariant Jordan curves in 
Section~\ref{subsec:ittproc}.

According to ``Sullivan's dictionary'' there is a close
correspondence between the dynamics of rational maps and of
Kleinian groups \cite{Suldic}. In
\cite{CanThu}  Cannon-Thurston constructed   Peano curves
related 
 to  the fundamental group of a
hyperbolic $3$-manifold $M^3$ that fibers over the
  circle. Theorem~\ref{thm:inv_peano} may be viewed as the
corresponding result in the case of rational maps. This  provides  
 another entry in Sullivan's dictionary.

There is a converse to Theorem~\ref{thm:inv_peano} (see
\cite{Me09b}); namely,  if for a Thurston map $f\colon S^2\to S^2$
there exists an iterate $F=f^n$ that has an $F$-invariant Peano
curve, then $f$ is expanding.  
\index{Thurston map!iterate of}
\index{iterate of Thurston map}
\index{F f@$F=f^n$}
\index{Thurston map!expanding}
\index{expanding}


In the early 1980s Douady and Hubbard observed  that
there are rational maps with Julia sets that ``contain'' the
Julia sets of some polynomials. This motivated  them to
introduce the notion of a \emph{mating of polynomials}. This 
operation combines two polynomials  geometrically,  often
giving  a rational map. In fact, Thurston's characterization  of
rational maps (Theorem~\ref{thm:Thurston}) was in part
motivated by the question when a map arising as a mating ``is''
a rational map. The notion of Thurston equivalence
also appears naturally  in this context. 

 There are many different
variants of matings. Here we will only define the one most relevant  
for us.  An  introduction to
matings can be found in \cite{Mi04} and  an overview of the different constructions  in \cite{MeyPet}.

Let   $P(z)=z^n+ a_{n-1}z^{n-1} + \dots +a_0$ be  a  monic  polynomial with complex coefficients, and $n=\deg(P)\ge 2$.  Then the \emph{filled Julia set} $\mathcal{K}= \mathcal{K}_P$ of $P$ is
the set of all points $z\in \C$ with bounded orbit
$\{P^n(z)\}_{n\in\N}$ in $\C$. We assume  that $\mathcal{K}$ is connected
and locally connected.  Then  there is a conformal  map $\phi\colon
\CDach\setminus \overline{\D}\to \CDach \setminus \mathcal{K}$
that satisfies $\phi(z^n)= P(\phi(z))$ 
for all $z\in \CDach\setminus \overline{\D}$ (this is {B\"{o}ttcher's theorem}; see
\cite[Section~9]{Mi} or \cite[Section II.4]{CG}). By {Carath\'{e}odory's
  theorem} (see, for example, \cite[Theorem~17.14]{Mi})  the map
$\phi$ extends to the unit circle $\Sph^1=\partial \D$. We call
the restriction $\sigma\colon
\Sph^1 \to \mathcal{K}$ of this extension to the unit circle   the \emph{Carath\'{e}odory loop}.  Then 
$\sigma(\Sph^1)= \partial\mathcal{K}= \mathcal{J}$ is the Julia set
of $P$. 

Since $\sigma$ is the extension of $\phi$, it follows
that $\sigma(z^d) = f(\sigma(z))$ for 
all $z\in \Sph^1$, i.e., the following diagram commutes:
\begin{equation}
  \label{eq:Cara_semi}
  \xymatrix{
    \Sph^1 \ar[r]^{z\mapsto z^d} \ar[d]_\sigma & \Sph^1 \ar[d]^\sigma
    \\
    \mathcal{J} \ar[r]^f & \mathcal{J}\rlap{.}
  }
\end{equation}

In general, the map $\sigma$ is not injective and so we only obtain 
 a {\em semi-conjugacy} here. 
 The existence of this  semi-conjugacy is one of the main reasons 
why  the dynamics of polynomials is much
better understood than the dynamics of arbitrary rational maps. In particular, it
can be used  to describe the dynamics of a polynomial on its  Julia set  in
combinatorial terms.

A \emph{(topological) mating}\index{mating} is now  defined as
follows. Let $P_\wt$ and $P_\bt$ be 
two monic  polynomials of the
same degree $d\ge 2$ (our use of the indices $\wt$ and $\bt$ is motivated 
by the close connection to the coloring of tiles as discussed in 
Section~\ref{sec:tiles}). 
We assume 
that their filled Julia sets   $\mathcal{K}_\wt$ and  $\mathcal{K}_\bt$ are connected and
locally connected (equivalently, one can impose these conditions on the Julia sets of the polynomials). Let $\sigma_\wt \colon \Sph^1 \to \mathcal{K}_\wt$
and $\sigma_\bt \colon \Sph^1\to \mathcal{K}_\bt$ be the corresponding 
Carath\'{e}odory loops. We now consider the (topological) disjoint union $\mathcal{K}_\wt
\sqcup \mathcal{K}_\bt$. Then a map  $P_\wt \sqcup P_\bt$ is naturally
defined on this set by letting it act on $\mathcal{K}_\wt$ as $P_\wt$ and on  $\mathcal{K}_\bt$ as $P_\bt$. 
Then  $P_\wt \sqcup P_\bt$  is clearly  a continuous map on  
$\mathcal{K}_\wt
\sqcup \mathcal{K}_\bt$. 

Let $\sim$ be  the equivalence relation  on
$\mathcal{K}_\wt \sqcup \mathcal{K}_\bt$ generated
by  (i.e., the smallest equivalence relation satisfying) the relation
\begin{equation}
  \label{eq:szszbar}
  \sigma_\wt(z)\sim \sigma_\bt(\bar{z})
\end{equation}
for  $z\in \Sph^1 =\partial \D$. Then the  \emph{mating} of
$\mathcal{K}_\wt$ and $\mathcal{K}_\bt$ is defined as $\mathcal{K}_\wt\mate
\mathcal{K}_\bt\coloneqq  \mathcal{K}_\wt \sqcup \mathcal{K}_\bt/\!\sim$.  
Moreover,  based on Lemma~\ref{lem:f_descends} 
 it follows from \eqref{eq:Cara_semi}
that the map
\begin{equation*}
  P_\wt \sqcup P_\bt\colon \mathcal{K}_\wt 
\sqcup \mathcal{K}_\bt \to \mathcal{K}_\wt \sqcup \mathcal{K}_\bt
\end{equation*}
descends to the quotient by $\sim$, i.e., to a map
\begin{equation*}
  P_\wt \mate P_\bt \colon \mathcal{K}_\wt \mate \mathcal{K}_\bt \to
  \mathcal{K}_\wt \mate \mathcal{K}_\bt.
\end{equation*}
This map is called the  \emph{(topological) mating} of $P_\wt$ and
$P_\bt$. 

The space $\mathcal{K}_\wt \mate \mathcal{K}_\bt$ may not be a
``nice'' topological space; in fact, it may not even be Hausdorff.
Surprisingly often though, the mating results in a map
$P_\wt\mate P_\bt$ that is topologically conjugate to a rational
map. This is particularly striking in cases when
$\mathcal{K}_\wt$ and $\mathcal{K}_\bt$ are \emph{dendrites} and
have no interior points.

The situation is best understood  for quadratic polynomials. 
The following statement is currently the best result  on the existence of
matings. 

\begin{theorem}
  \label{thm:mating_quadratic}
  Let $P_{\wt}(z)= z^2 + c_{\wt}$ and $P_{\bt}(z)= z^2 + c_{\bt}$ be
  post\-criti\-cally-finite quadratic polynomials such that
  $c_{\wt}$ and $c_{\bt}$ are not contained in conjugate limbs
  of the Mandelbrot set. Then the mating $P_{\wt} \mate P_{\bt}$
  is topologically conjugate to a (postcritically-finite)
  rational map. 
\end{theorem}

For the terminology and the proofs see \cite{Le92}, \cite{Re92},
and \cite{Sh00}. 
Instead of asking when two polynomials can be mated, one can also investigate  when  a rational 
 map $f$ arises as a mating of two polynomials
$P_{\wt}$ and $P_{\bt}$. If this is  the
case, we say that $f$ \emph{unmates} into $P_{\wt}$ and
$P_{\bt}$. The following result about unmatings was  established 
in \cite{Me09c}. An
overview of the construction, as well as several examples, can 
be found in \cite{Me14}.

\begin{theorem}
  \label{thm:exp_Th_matings}
  \index{Thurston map!iterate of}
  \index{iterate of Thurston map}
  \index{F f@$F=f^n$}
  Let $f\colon S^2\to S^2$ be an expanding Thurston map without
  periodic critical points. Then every sufficiently high iterate
  $F=f^n$ is topologically conjugate to the mating of two
 monic  polynomials $P_\wt$ and $P_\bt$. 
\end{theorem}

 Here the polynomials $P_\wt$ and $P_\bt$ are  postcritically-finite  and have  the same degree  as
$F$. Their  Julia sets  are dendrites. 

Theorem~\ref{thm:inv_peano} and Theorem~\ref{thm:exp_Th_matings}
are closely related. It is not hard to see that Theorem~\ref{thm:exp_Th_matings} implies
Theorem~\ref{thm:inv_peano}, but for  the proof one actually  first 
 establishes  the latter  theorem and then derives the former 
  as a consequence.

We also remark that there is a version of
Theorem~\ref{thm:exp_Th_matings} for expanding Thurston maps
that do have periodic critical points. 

In \cite{Me14} a sufficient criterion was given when an
expanding Thurston map $f$ unmates into two polynomials  and an algorithm was  provided 
    to determine
the  polynomials. However, the criterion  was not necessary.

\begin{prob}
  \label{prob:f_unmates}
  Is it possible to give   a necessary and sufficient condition when an expanding
  Thurston map unmates into two polynomials? Can the polynomials be determined  by an algorithm in this case? 
 \end{prob}

\subsection*{The iterated monodromy group}
\label{sec:iter-monodr-group}
An important question is how to decide when two given Thurston
maps are equivalent. A prominent example where this is relevant
and hard to decide is in the following situation. There are
exactly three distinct (up to conjugacy by a M\"obius
transformation) quadratic polynomials whose critical point is
periodic with period $3$; each such map has exactly four
postcritical points (including $\infty\in \CDach$). These maps
are known as the ``rabbit'', the ``anti-rabbit'', and the
``airplane''. If we postcompose such a map with a Dehn twist about a
Jordan curve separating two of the postcritical points from the
other two, then  this results in a Thurston map that is
(orientation-preserving) Thurston equivalent to one of these three
maps. Deciding whether the map is equivalent to the rabbit, the
anti-rabbit, or the airplane is known as the \emph{twisted rabbit
  problem}. It was solved by Bartholdi and Nekrashevych \cite{BN}
by using the concept of the {\em iterated monodromy group}.  This
is an important group associated with every Thurston map. It was
first considered by Kameyama \cite{Ka03a} and systematically
studied by Nekrashevych \cite{Ne} in a more general setting.

To define this group, we consider 
 a Thurston map $f\colon S^2\to
S^2$ and  a point $p\in S^2\setminus \post(f)$. For each $n\in \N_0$ let
$V_n\coloneqq  f^{-n}(p)$ be the preimage set of $p$ under $f^n$.   Then the  \emph{preimage tree} of $p$ with respect to
$f$ is the graph $T=(V,E)$, whose  set of vertices is the disjoint union
 $V= \bigsqcup_{n\in \N_0} V_n$. Moreover, if  $x\in V_n\sub V$  
with $n\geq 1$,  then we connect the  vertices $x$ and
$f(x)\in V_{n-1}$ by an edge and  all edges arise  in this way. It is clear that the graph 
$T$ is indeed a tree. 

Now  consider a loop $\gamma\subset S^2\setminus \post(f)$
starting (and ending) at $p$.  Then  $\gamma$ represents an element
$g=[\gamma]$ in the fundamental group  $G=\pi_1(S^2\setminus \post(f), p)$ of 
$S^2\setminus \post(f)$. If 
 $x\in f^{-n}(p) = V_n \subset V$,  then $\gamma$ can be
lifted by  $f^n$ to a path  $\gamma_x$ starting at $x$. The endpoint $y$
of $\gamma_x$ will also belong to  $V_n\subset V$. By the homotopy
lifting theorem (\cite[Proposition~1.30]{Ha}) $y$ depends only on the homotopy class $g=[\ga]$ of 
$\gamma$.  If we set
$g(x) =y$, then one can show  that  $g$ induces   an automorphism $\varphi(g)$ 
of the tree $T$. This defines  an action of $G$  on $T$,  and, if we denote by $ \Aut(T)$ 
 the automorphism group  of $T$, we obtain a group homomorphism 
$\varphi
\colon G \to \Aut(T)$. The
{\em iterated monodromy group}   $\img(f)$ of $f$   is defined as the quotient of $G$ that acts effectively on $T$, or more precisely, 
\begin{equation*}
  \img(f) 
  \coloneqq  
  G/\ker(\varphi) 
  \cong 
  \varphi(G), 
\end{equation*}
where $\ker(\varphi)$ denotes the kernel of $\varphi$.

The iterated monodromy group is invariant under Thurston equivalence in the sense that two equivalent  Thurston maps have the same iterated monodromy group up to isomorphism.  In fact, with some additional data, it is a
complete invariant for  Thurston equivalence  (see
\cite[Theorem~6.5.2]{Ne}). The solution of the twisted rabbit problem  by Bartholdi and 
Nekrashevych was based on this fact.



Iterated monodromy groups are  \emph{self-similar groups} (see
\cite{Ne}). They can be quite complicated, even  for very simple maps. 
In general, their algebraic properties (such  as torsion or amenability) are poorly 
understood.  Here we will only discuss one particularly interesting aspect of iterated monodromy groups in  more detail, namely their growth behavior.   We first have to recall 
 some definitions.

 Let $S$ be a finite and symmetric set of generators for a finitely-generated 
group $G$. For  $g\in G$ let $\ell(g)=\ell_{G,S}(g)$ be the
minimal length of a word in the alphabet  $S$ that represents $g$. This is  equal to 
the distance of $g$ from the neutral element $e$ of $G$ in the Cayley graph
$\mathcal{G}(G,S)$ (recall these concepts from  Section~\ref{sec:Cannconj}). 

Let $N= N_{G,S} \colon \N_0 \to \N$ be the {\em growth
function} of $G$  given by $$N(n) =  \# \{g\in G : \ell(g) \leq n\} \quad \text{for $n\in \N_0$.} $$ We
say that $G$ is of \emph{polynomial growth} if $N(n) $ is
bounded from above by a polynomial in $n$, and  of 
\emph{exponential growth}
if $N(n)$ is bounded from below by an exponential
function of the form $C\exp(\alpha n)$ with  $C,\alpha>0$.
If $G$ is neither of polynomial nor of exponential growth, then 
we say that it is of \emph{intermediate growth}. The growth behavior of $N$ 
is independent of the choice of the generating set $S$, and can
therefore be considered as a property of the group $G$. 
 
For example, free groups and fundamental groups of closed
hyperbolic manifolds are of exponential growth. A celebrated
theorem due to Gromov says that a group is of polynomial growth
if and only if it is {virtually nilpotent} (see \cite{Gr81} and \cite{Kl10}). This answered a question raised by Milnor
in  \cite{Mi68}. In the same note, Milnor asked whether groups of
intermediate growth actually exist. First examples of such
groups were later found by Grigorchuk \cite{Gr84}.
 
It is quite striking that iterated monodromy groups of very
simple rational maps can be groups of intermediate growth. For
example, this is the case for $\img(P)$ where 
$P(z)=z^2+\iu$ (see \cite{BP06}).  While intermediate growth of the iterated
monodromy group has been shown for some other polynomials,
at present no general sufficient condition for this to be true  is  known.


If  a postcritically-finite quadratic polynomial $P$ 
has two distinct Fatou components whose closures 
 intersect in a single point, then it is not hard
to show  that its iterated monodromy group is of exponential
growth. For example, this is true for $P(z)=z^2-1$.
 

Latt\`{e}s maps and Latt\`{e}s-type maps  have iterated monodromy groups that are virtually isomorphic to $\Z^2$. 
Apart from some  special cases such as the examples discussed, very little is known in general about the growth 
of iterated monodromy groups of postcritically-finite polynomials, and even 
 less   for Thurston maps that are non-polynomial (i.e., not Thurston polynomials; see Section~\ref{sec:further-results}). Some 
 examples of non-polynomial Thurston
maps that  have iterated monodromy groups of
exponential growth were found  in \cite{HM}. 
 
\begin{prob}
  \label{prob:growth_img}
  Are there non-polynomial Thurston maps with 
  iterated mono\-dromy groups  of intermediate growth? 
\end{prob}

\subsection*{Ergodic theory of expanding Thurston maps}
\label{sec:ergod-theory-expand}
The ergodic theory of expanding Thurston maps was 
developed further by Zhiqiang Li. In \cite{Li1} it was shown that
the measure of maximal entropy of an expanding Thurston map can
be obtained as a weak{$^*$}-limit of point masses at periodic
points, or at 
preimages of any point. Similar results for
rational maps had  been established   before  by Lyubich \cite{Ly83}. 

Li also investigated  \emph{equilibrium
  states} for 
expanding Thurston maps \cite{Li2}. These are measures obtained from 
H\"{o}lder continuous functions, called {\em potentials}.   For the precise definition, let $f\: S^2\ra S^2$
be an expanding Thurston map and $\phi\: S^2\ra \R$ be a H\"older continuous function. Here $S^2$ is equipped with a visual metric for $f$. 
We define the {\em topological pressure} of  $\phi$  with respect to $f$  as 
\begin{equation}\label{eq:toppre} 
 P(\phi, f) =\sup_{\mu}\, \biggl \{h_{\mu}(f) + \int \phi \,d\mu\biggr\}, 
\end{equation} 
where the supremum is taken over all $f$-invariant (Borel) probability measures on $S^2$ (recall that 
$h_\mu(f)$ denotes the measure-theoretic entropy of $f$ with respect to $\mu$; see Section~\ref{sec:reviewmdyn}).

 An $f$-invariant  measure $\mu_{\phi}$ for  which the supremum in \eqref{eq:toppre} is
attained  is called an \emph{equilibrium state}. Li showed  the
existence and uniqueness of equilibrium states for any
H\"{o}lder continuous potential $\phi$. These
measures can also be described as weak{$^*$}-limits of suitably
weighted point masses at  periodic points or preimage points of a given point.  

Bowen
\cite{Bo72}  introduced 
the concept of an \emph{$h$-expanding} map and Misiurewicz \cite{Mi76}  the weaker notion of  an \emph{asymptotically
  $h$-expanding} map. Roughly
speaking,  these notions mean that the map is  expanding in a
 strong 
sense except on a set of topological entropy $0$.
We will not give the precise definitions here, because they are
somewhat technical. Li showed \cite{Li4} that no expanding
Thurston map is $h$-expanding and that an expanding Thurston map
is asymptotically $h$-expanding if and only if it has no
periodic critical points. A comprehensive account of Li's work
on the ergodic theory of expanding Thurston maps can be found in
\cite{Li3}.

For a rational expanding Thurston map
$R\colon \CDach \to \CDach$ there exists a unique $R$-invariant (Borel) measure $\lambda_R$ that is absolutely continuous 
  with respect to Lebesgue measure 
 $\leb$ on $\CDach$ (see
Theorem~\ref{thm:ex_inv_abs_L} and
Section~\ref{sec:Ruelle}).  Suppose 
that an expanding Thurston map $f\colon S^2\to S^2$
is topologically conjugate to a rational map
$R\colon\CDach \to \CDach$, i.e., there is a homeomorphism
$h\colon S^2\to \CDach$ such that $f= h^{-1}\circ R\circ h$.
Then we can  pull back the measure $\lambda_R$ by $h$ to obtain a measure $\lambda_f$ on $S^2$. More explicitly, $\lambda_f$
is defined by setting $\lambda_f(A) = \lambda_R(h(A))$ for each  Borel set
$A\subset S^2$. One can show that  
the measure $\lambda_f$ only depends on $f$ and not on the choice
of the map $h$ that conjugates $f$ to a rational map 
(this can be derived from the uniqueness statement in
Theorem~\ref{thm:ex_inv_abs_L}). 

Often   $f$ is
known to be topologically conjugate to a rational map, even though we do not have an explicit 
conjugating map $h$. A simple example for this situation is  when
the expanding Thurston map  $f$ has precisely three  postcritical points, but no
periodic critical points (see
Theorem~\ref{thm:3postrat}~\ref{item:3post_exp}).

\begin{prob}
  \label{prob:lambda_intr}
  Assume an expanding Thurston map $f\colon S^2\to S^2$ is
  topologically conjugate to a rational map. Is it possible to
  construct   the measure $\lambda_f$ intrinsically? 
\end{prob}

In other words, we would like to obtain 
the measure  $\lambda_f$ without the use of the
conjugating map $h$.  If one can find a good characterization of the measure $\lambda_f$
if it exists, it might be possible to decide whether $f$ is
topologically conjugate to a rational map by measure-theoretic
methods. This is related to Thurston's theorem.

\ifthenelse{\boolean{singlechapter}}{ 

\appendix 
%
%

\begin{appendix}

\chapter{}
\label{Appendix}
In this appendix we collect various facts whose discussion would
have interrupted the main flow of our presentation. Among other
things we discuss branched covering maps 
(Section~\ref{sec:appbracovmap}) and orbifolds
(Sections~\ref{sec:orbifolds-coverings}
and~\ref{sec:expratThmaps}) in quite some detail, because it is
hard to find the statements relevant for us in the literature.

\section{Conformal metrics}
\label{sec:metrspterm}
 
 Here we summarize some standard metric space terminology and record facts  related to conformal metrics.
 
Let $(X,d)$ be a metric space.  
A {\em path} $\ga$ in $X$ is a continuous 
map $\ga\: [a,b]\ra X$ defined on some interval $I=[a,b]\sub \R$. Sometimes one considers also paths defined on half-open or open intervals $I\sub \R$. 
As is common, we also use the notation $\ga$ for the image set $\ga(I)\sub X$ of a path. 
 
A  path $\ga\:[a,b]\ra X$ {\em joins} two points 
$x,y\in X$ if $\ga(a)=x$ and $\ga(b)=y$.  
The {\em length} of  $\ga$ is given as 
$$ \length_d(\ga)\coloneqq  \sup \sum_{k=1}^n d(\ga(t_{k-1}),\ga(t_k))\in [0,\infty], $$
where the supremum is taken over all $n\in \N$ and all points $t_0=a<t_1<\dots<t_n=b$. The path $\ga$ is called {\em rectifiable} (with respect to $d$) if $L\coloneqq \length_d(\ga)<\infty$. 
In this case, we  define $L(t): =\length_d(\ga|[a,t])$ for $t\in [a,b]$. Then there is a unique path 
$\widetilde \ga\: [0,L]\ra X$, called the {\em arclength parametrization} of $\ga$,  such that 
$\widetilde \ga(L(t))= \ga(t)$ for  $t\in [a,b]$. 

If $\rho\: X\ra [0,\infty]$ is a Borel function, we define 
the path integral of $\rho$ along the rectifiable path $\ga$ as 
$$ \int_{\ga}\rho\, ds: = \int_0^L\rho(\widetilde  \ga(s))\, ds. $$
The metric   $d$ is called a 
{\em length  metric}\index{length!metric}\index{metric!length}
or {\em path  metric}\index{path!metric}\index{metric!path}
if 
$$
d(x,y)\coloneqq \inf_{\gamma} \length_d(\gamma)
$$
for  $x,y\in X$, 
where the infimum is taken over all  paths $\gamma$ in $X$ joining $x$ and $y$. The 
metric is a 
{\em geodesic metric}\index{geodesic metric}\index{metric!geodesic} if this  infimum is attained as a minimum. A path realizing 
this  infimum  is called a {\em geodesic segment} joining $x$ and $y$. 

Suppose $U$ is a region in the complex plane $\C$ equipped with the Euclidean metric, and $\rho\:U\ra (0,\infty)$ is a positive and continuous function on $U$.
Then we can define a metric on $U$ by setting 
\begin{equation}\label{eq:confmetr}
 d(u,v)= \inf_\ga  \int_{\ga}\rho(z)\, |dz|
\end{equation} 
for $u,v\in U$, where the infimum is taken over all rectifiable paths $\ga$  in $U$ joining 
$u$ and $v$ and $|dz|$ refers to integration with respect to Euclidean arclength. 
We say that $d$ is the 
{\em conformal metric}\index{conformal metric}\index{metric!conformal}
on $U$ with 
{\em length element}\index{length!element}
$ds=\rho(z)\, |dz|$ and call $\rho$ the {\em conformal factor} of $d$.

Let $\D\coloneqq \{z\in \C: |z|<1\}$ be the unit disk in the complex plane
 $\C$. Then the 
{\em hyperbolic metric}\index{hyperbolic!metric}\index{metric!hyperbolic} 
$d_0$ on $\D$ is defined as the conformal metric with length element 
\begin{equation}
  \label{eq:def_hyp}
  ds= \frac{2|dz|}{1-|z|^2}.  
\end{equation}
The space $(\D, d_0)$ is geodesic and a model of the hyperbolic
plane $\mathbb{H}^2$.  The conformal automorphisms of $\D$ are precisely the
orientation-preserving isometries of $(\D, d_0)$.

Similarly, if $\mathbb{H}=\{ z\in \C:\imag(z)>0\}$ is the upper half-plane in $\C$, then the 
hyperbolic metric on $\mathbb{H}$ is given by the length element 
\begin{equation}
  \label{eq:def_hyponH}
  ds= \frac{|dz|}{\imag(z)}.  
\end{equation}
If we equip $\mathbb{H}$ with this metric, then $\mathbb{H}$ and $(\D,d_0)$ are isometric spaces.

The Riemann sphere  $\CDach=\C\cup\{\infty\}$ can be equipped with two natural 
metrics that are essentially equivalent. The 
{\em spherical metric}\index{spherical!metric}\index{metric!spherical}
is a conformal  metric 
on $\Cdach$  given by the  length element 
 \begin{equation}\label{eq:sphlel} 
  d\sigma=\frac{2|dz|}{1+|z|^2}
  \end{equation} 
 (strictly speaking, this gives the restriction of the spherical metric to $\C$). 
 This is actually a geodesic metric on $\Cdach$. 
 
 One can  identify $\Cdach$  with the unit sphere in $\R^3$ via
stereographic projection. The \emph{chordal
  metric}\index{chordal metric|textbf}\index{metric!chordal|textbf}\index{S.@$\sigma$}
$\sigma$ on $\CDach$ is the metric that corresponds to the
Euclidean metric in $\R^3$ under  this identification. More explicitly, 
\begin{equation}
  \label{eq:def_chordal}
  \sigma(z,w)= \frac{2\abs{z-w}}{\sqrt{1+\abs{z}^2}\sqrt{1+\abs{w}^2}}
\end{equation}
for $z,w\in \C$, and 
$$\sigma(\infty,z)=\sigma(z,\infty)=  \lim_{w\to \infty}
\sigma(z,w)= \frac{2}{\sqrt{1+ \abs{z}^2}}$$ 
for $z\in \C$. 

The spherical and the chordal metrics on $\Cdach$ are comparable up to a uniform 
factor that approaches $1$ for small distances. Accordingly, the length of paths are the same for both metrics. Usually, we equip $\Cdach$ with the chordal metric $\sigma$ 
and consider the spherical metric as an ``infinitesimal version'' of $\sigma$.

Let $f\: U\ra \CDach$ be a holomorphic map
on a region $U\sub \CDach$.  Then 
its expansion with
respect to the chordal metric is measured by the 
{\em spherical derivative}\index{spherical!derivative}.
It is given by
\index{$f^{\sharp}$} 
\begin{equation}
  \label{eq:defsphericaldf}
  f^{\sharp}(z)
  \coloneqq\lim_{w\to z} \frac{\sigma(f(w), f(z))}{\sigma(w,z)}
  =  \frac{1+|z|^2}{1+|f(z)|^2}|f'(z)|
\end{equation}
for  $z\in \CDach$. If $z,f(z)\in \C$, 
then $f'(z)$ denotes 
the derivative of $f$ at $z$ as usual. If 
 $z=\infty$ or $f(z)=\infty$, then the  last
expression in \eqref{eq:defsphericaldf} has to be understood as a suitable limit.

Similarly, let  $f\: U\ra V$
be a holomorphic map 
between regions  $U,V\sub\C$. Suppose $U$ and $V$ 
are  equipped with length metrics 
$d$ and 
$\widetilde d$  
induced by conformal factors 
$\rho$ and  $\widetilde \rho$,
respectively.  Then the  distortion of these  metrics by $f$ at a point $z\in U$ is measured by    
\begin{equation}\label{eq:gennorm}
\Vert f'(z) \Vert_{\rho, \widetilde \rho} \coloneqq  \lim_{w\to z} \frac {{\widetilde d}(f(w), f(z))}{d(w,z)} = \frac {\widetilde \rho( f(z)) } {\rho(z)} |f'(z)|. 
\end{equation} 
Here we simply write $\Vert f'(z) \Vert$ if the metrics and their conformal factors are clear from the context. 
  
If $f\: U \ra \CDach$ is
differentiable at a point $p$ of a region $U\sub \CDach$, but not necessarily holomorphic, then  $Df(p)$ stands for the derivative of $f$ at $p$, considered as a linear map between the tangent spaces
at $p$ and $f(p)$.  We use  $\Vert Df(p)\Vert_\sigma$  to indicate  the
norm of $D(p)$ with respect to the spherical metric; if
$p, f(p)\in \C$ and $\Vert Df(p)\Vert$ denotes the Euclidean
norm, then \index{$\norm{Df}_\sigma$}
\begin{equation}
  \label{eq:sphder2} 
  \Vert Df(p)\Vert_\sigma
  = 
  \frac{(1+|p|^2)\Vert Df(p)\Vert}{1+|f(p)|^2}. 
\end{equation} 
In  case  $f$ is holomorphic, this agrees with the
spherical derivative of $f$ at $p$.

If $U\sub \CDach$ is a region and $\rho$ is a positive
continuous function on $U$, then one can define a conformal
metric with length element $ds=\rho\,d \sigma$ as in
\eqref{eq:confmetr}, where we use integration  with respect to 
spherical arclength $d\sigma$  as in \eqref{eq:sphlel}
  instead of $|dz|$. It is convenient 
   (see Sections~\ref{sec:expratThmaps}) to allow 
{\em singular conformal metrics}\index{conformal metric!singular}\index{singular conformal metric} 
with a  continuous conformal factor $\rho\colon U\setminus P \to(0,\infty)$, where $P\sub U$ is a discrete set in $U$ (i.e., it has no limit points in $U$) such that for each $p\in P$ we have $\rho(z) \to 0$ or $\rho(z) \to
\infty$ as $z\to p$. If
 in the latter case $\rho(z) \lesssim
\sigma(z,p)^{-\alpha}$ for $z$ near $p$ with  $\alpha<1$, then the conformal metric with length element $ds=\rho\, d\sigma$ on $U\setminus P$  extends to a
length metric on $U$.

\section{Koebe's distortion theorem}
\label{sec:Koebe}

In this section we discuss some distortion estimates for conformal maps that can be derived from the classical Koebe distortion theorem. Since the  conformal maps we are interested in  are usually  defined on subregions  of the Riemann sphere, 
it is most natural  to formulate the  estimates in terms of
spherical derivatives  and chordal distances (see Section~\ref{sec:metrspterm}).  This  mostly amounts to a straightforward translation of the corresponding classical distortion estimates for the Euclidean  metric with one {\em caveat}: to get uniform control 
for the constants, it is important to require that the image region of the 
conformal map  is not too large in the Riemann sphere. 
We will impose the condition  that  the image region of the map  is contained in a
 hemisphere of $\CDach$, i.e., a chordal disk of radius $\sqrt 2$. 

 \begin{theorem}[Spherical version of Koebe's distortion theorem]
   \label{thm:Koebe}
   \index{Koebe distortion}
   \index{distortion}
Suppose that
   $0<r<R< \diam_\sigma(\CDach)=2$, $z_0\in \CDach$,
   $B\coloneqq B_\sigma(z_0,R)\subset \CDach$, and  $g\colon B\to
   g(B)\subset \CDach$ is a conformal map such that its image
   $g(B)$ is 
   contained in a hemisphere of $\CDach$. Then for all $z,w\in  \widetilde B \coloneqq 
   B_\sigma(z_0,r)$ we have 
   \begin{align}
     \label{eq:Koebe_dg}
     g^\sharp(z) & \asymp g^\sharp(w) \quad \text{and}\\
     \label{eq:Koebedzw}
    {\sigma(g(z),g(w))} &\asymp g^\sharp(z) \sigma(z,w).
      \end{align} 
      Here $C(\asymp)= C(r/R)$, and $C(r/R)\to 1$ as $r/R\to
      0$.

      Moreover, there exist two  constants $c_1=c_1(r/R)>0$ and
      $c_2=c_2(r/R)>0$ such that for $w_0\coloneqq  g(z_0)$ we have 
   \begin{equation}
     \label{eq:Koebe14}
     c_1  B_\sigma(w_0, g^{\sharp}(z_0) r)
     \sub  
    g(B_\sigma(z_0,r))
    \sub 
    c_2 B_\sigma(w_0, g^\sharp(z_0 ) r).
   \end{equation}
   \end{theorem}
Here we use the notation
$\lambda B_{\sigma}(w_0, r)\coloneqq  B_\sigma(w_0, \lambda r)$ for $\lambda> 0$.

  If we use   the Euclidean metric and the usual derivative, then for conformal maps defined on Euclidean disks and with  images in $\C$ the statements \eqref{eq:Koebe_dg},  
 \eqref{eq:Koebedzw},   and \eqref{eq:Koebe14}   immediately follow 
 from  the classical\index{Koebe distortion}
 Koebe distortion theorem (see \cite[Theorem~1.3 and Corollary~1.4]{Po}).  
 
  By considering  $g(z)=nz$ with $n\in \N$  on the unit disk  $\D$ one can  see that statement \eqref{eq:Koebe_dg}, for example, is not true with a constant independent of 
  the map if one does not impose some restriction on its image.
 
 \begin{proof}
We will derive  the spherical versions of the distortion statements   from the Euclidean versions as follows. 
 
By pre- and postcomposing the map $g$ with auxiliary rotations of $\CDach$ (which can be realized by M\"obius transformations), we may assume that $z_0=0$ and $g(B)\sub \D$. 
Let  $r'$ and $R'$ be the Euclidean radii of the 
disks $\widetilde B=B_\sigma(0,r)$ and $B=B_\sigma(0,R)$,
respectively. Then it follows from \eqref{eq:def_chordal} that 
\begin{equation}
  \label{eq:rrRR}
  r'=r/\sqrt{4-r^2}\text{ and } R'=R/\sqrt{4-R^2}.  
\end{equation}
Since  $R\leq 2$, and so  $r\le 2r/R$, we obtain the estimate 
\begin{equation} \label{eq:Euclradcontr}
 r'\le \frac{r}{2\sqrt{1- (r/R)^2}} \leq \frac{1}{\sqrt{1-(r/R)^2}}.
 \end{equation}
 Thus on $\widetilde{B}$ the chordal and the Euclidean metrics
 differ by a multiplicative constant only depending on $r/R$.  On
 $g(B)\subset \D$ chordal and Euclidean metrics differ by a
 uniform multiplicative constant. 

 Note that \eqref{eq:rrRR} gives 
 \begin{equation} 
  \label{eq:ratiocontr}  
  r'/R'\le r/R.
\end{equation}  
Thus the statements \eqref{eq:Koebe_dg}--\eqref{eq:Koebe14}
follow from their Euclidean counterparts. 
Moreover, in  \eqref{eq:Koebe_dg} and  
\eqref{eq:Koebedzw}  we  indeed have 
$C(r/R)\to 1$ as $r/R\to 0$ which again easily follows from the
Euclidean counterpart of this statement.  
\end{proof} 
 
 Of course, for a single conformal map $g$ we always have an estimate as  in \eqref{eq:Koebe_dg}
 if we allow the constant also to depend on the map, because the spherical derivative of a conformal map is a positive continuous function.  In applications 
 we often consider families of maps where all maps have images
 contained in a hemisphere with possibly finitely many exceptions. Then one still obtains an estimate as in \eqref{eq:Koebe_dg}  with a uniform constant for the whole family if one uses the uniform constant in  \eqref{eq:Koebe_dg} and adjusts it so that the estimate remains valid also for the  finitely many exceptional maps.

We also require   a version of  Koebe's distortion theorem  
for conformal maps on multiply connected subregions of $\CDach$.   To get uniform distortion estimates, we again assume that     the image of the conformal map is  contained in a
 hemisphere.

\begin{lemma} 
  \label{lem:distest} 
   \index{Koebe distortion}
   Let $\Om\sub \CDach$ be a region, and $A,B\sub \Om$ be compact
   sets each consisting of at least two points. 
   Then for each conformal map
   $h\: \Om\ra \Om'\coloneqq h(\Om)\subset \CDach$ whose image
   $\Om'$ is contained in a hemisphere of $\CDach$ we have
\begin{align}
\label{eq:dhdiamhA}
   \diam_\sigma (h(A)) &\asymp h^\sharp (a)   \text{ and}  
  \\
  \label{di1}
  \dist_\sigma(h(A), \partial \Om')&\gtrsim   h^\sharp (a)  
   \text{ for each $a\in A$,  }\\
   \intertext{where $C(\asymp)=C(A,\Om)$ and $C(\gtrsim)=C(A,\Om)$. Moreover, }
  \label{di0}
    \diam_\sigma(h(A))& \asymp  \diam_\sigma(h(B)),     \end{align} 
  where $C(\asymp)=C(A, B, \Om)$.
\end{lemma}
 
The main point in Lemma~\ref{lem:distest} is that under the given assumptions the constants in the inequalities are independent of $h$.

 \begin{proof}   
In the following,  all metric notions refer to the chordal  metric $\sigma$.  If $D=B(z_0,r)$ is a  disk, we use the notation   $2D=B(z_0,2r)$
for the disk with the same center and twice the radius. Note that $\CDach\setminus \Om$ 
necessarily contains more than two points; so if $2D=B(z_0, 2r)\sub \Om$, then 
$2r<2=\diam(\CDach)$, and we can apply the distortion estimates of 
Theorem~\ref{thm:Koebe} for the disks $\widetilde B=D$ and $B=2D$.

A 
{\em Harnack chain}\index{chain!Harnack}\index{Harnack chain}
(in $\Om$) is a sequence $D_1, \dots, D_n$ of  disks  with $2D_i
\sub \Om$ for $i=1, \dots, n$ and   $D_{i}\cap D_{i+1}\ne \emptyset$ for $i=1, \dots, n-1$. We call $n$ the length of the Harnack chain, and say that it joins two points $u,v\in \Om$ if $u\in D_1$ and $v\in D_n$. 
Note that if $u$ and $v$ are two points in $\Om$ that can be joined by a Harnack chain of length 
$n$, then repeated application of \eqref{eq:Koebe_dg}  leads to $h^\sharp(u)\asymp h^\sharp(v)$ with 
$C(\asymp)=C_0^n$, where $C_0\ge 1$ is a universal constant. 

Now any two points in a compact subset $K$ of $\Om$ can be joined by a Harnack chain 
whose length is uniformly bounded only depending on $K$ and $\Om$.
This implies that 
\begin{equation}
  \label{eq:Koebe2_dh}
  h^\sharp(u) \asymp h^\sharp(v) \quad \text{for all $u,v\in A$,}
\end{equation}
 where $C(\asymp) = C(A,\Omega)$ is
independent of $u$, $v$, and  $h$.

Let  $a,u,v\in A$ be arbitrary. Then there exists a Harnack chain in $\Om$ that joins
$u$ and $v$ and has length uniformly bounded from above. Then \eqref{eq:Koebedzw}, the triangle inequality, and \eqref{eq:Koebe2_dh} 
give 
$$ \sigma(h(u),h(v))\lesssim h^\sharp(a)$$ 
with $C(\lesssim)=C(A,\Om)$.
Hence $\diam(h(A))\lesssim h^\sharp(a)$ with an implicit multiplicative constant only depending on $A$ and $\Om$. 
This gives one of the estimates  in  \eqref{eq:dhdiamhA}. 


To show the other estimate  in \eqref{eq:dhdiamhA}, we fix two distinct points $u_0,u_1\in A$ and a disk $D$ centered at $u_0$ with   $2D\sub \Om$ and $u_1\not\in D$. Then $h(u_1)\not\in h(D)$. So by the first inclusion in  
\eqref{eq:Koebe14} and by \eqref{eq:Koebe2_dh}, we have 
$$  \diam(h(A))\ge \sigma(h(u_0),h(u_1))\gtrsim h^\sharp(u_0)
\asymp h^\sharp(a)$$
for each $a\in A$ with implicit multiplicative constants only depending on $A$ and $\Om$.

To prove \eqref{di1}, we note that by  \eqref{eq:Koebe14} and \eqref{eq:Koebe2_dh} we have 
\begin{equation*}
  \dist(h(u), \partial \Omega')\gtrsim  h^\sharp(u) \dist(u,\partial \Omega)\asymp h^\sharp(a)
\end{equation*}
for all $u,a\in A$, where $C(\asymp)=C(A,\Om)$. Inequality  \eqref{di1} follows. 

Finally, in order to establish \eqref{di0} pick $a\in A$ and $b\in B$. 
 By similar  arguments as above one sees that 
$\diam(h(B))\asymp |h^\sharp(b)|$ with 
$C(\asymp)=C(B,\Om)$, and that $|h^\sharp(a)|\asymp |h^\sharp(b)|$ with $C(\asymp)=C(A,B,\Om)$.
Hence 
$$ \diam(h(A))\asymp |h^\sharp(a)|\asymp |h^\sharp(b)|\asymp
\diam(h(B))$$ 
with 
$C(\asymp)=C(A,B,\Om)$ as desired.  
\end{proof}

\section{Janiszewski's lemma}
\label{sec:Janis}

In this section we discuss some  topological facts related to separation of sets. We will  establish  Lemma~\ref{lem:ABKsep} that is required for the proof of Theorem~\ref{thm:Qianthm}.   

In the following, $S^2$ is  a
topological $2$-sphere.
If $\Om\sub S^2$ is a region, and $A,B\sub \Om $, then a set
$K\sub \Om $ {\em separates} $A$ and $B$  in $\Om$
if for every path $\ga$ in $\Om$ joining $A$ and $B$ (i.e.,
$\ga$ has one endpoint in $A$ and one in $B$), we have
$\ga \cap K \ne \emptyset$. Note that this is meaningful even if
$K$ is not disjoint from $A$ or $B$. We omit the phrase ``in $\Om$'' if
$\Om$ is understood.

 The set $K$ {\em
  separates} $x\in \Om$ and $y\in \Om$ (or $x\in \Om$ and a set
$B\sub \Om$) if $K$ separates $A=\{x\}$ and $B=\{y\}$ (or
$\{x\}$ and $B$) in $\Om$.   
 
The following fact is well known. 
 
\begin{lemma}[Janiszewski's lemma] 
  \label{lem:Janiszewski}
  \index{Janiszewski's lemma}
  Let $K,L\sub \R^2$ be two closed sets with
  $K\cap L=\emptyset$. If two points $x,y\in \R^2$ are separated
  by $K\cup L$, then they are separated by $K$ or by $L$.
\end{lemma}

A version of this can be found in \cite[Theorem~III.4.A]{Bi};
exactly the same statement is true if $\R^2$ is replaced with a
$2$-sphere $S^2$.

We need a  more sophisticated lemma in the same spirit. 

\begin{lemma}
  \label{lem:ABKsep} 
  Let $\Om\sub S^2$ be a simply connected region, 
$A,B\sub \Om$ be connected sets, and $K\sub \Om $ be a set 
that is relatively
  closed in $\Om$ and has finitely many connected components. If
  $K$ separates $A$ and $B$ in $\Om$, then one of the components
  of $K$ separates $A$ and $B$ in $\Om$.
\end{lemma} 
 
This is a rather straightforward consequence of
Jani\-szew\-ski's lemma if we make the additional assumption
that the sets $A$ and $B$ do not meet $K$. This assumption is
not convenient in our application of the lemma though (see the proof of 
Lemma~\ref{lem:size_separate}).
Possible crossings of the sets $A$ or $B$ with $K$ complicate
the situation and require a somewhat more involved argument.

 
\begin{proof} 
  If $\Om\ne S^2$, then $\Om $ is homeomorphic to $\R^2$, and,
  as we have a purely topological statement, we may assume
  $\Om=\R^2$.  We will first present the proof in this case, and
  comment on the minor changes necessary for $\Om=S^2$ after the
  argument.
  
  We first establish the statement under an additional
  hypothesis.
  
\smallskip 
{\em Special Case:} The set $B$ is a singleton set, i.e., $B=\{b\}$, where $b\in \R^2$. 

\smallskip
We run an induction on the number $n$ of components of $K$.  The
induction beginning $n=1$ is clear. For the induction step, we
assume that the statement is true (for given $A$ and $B$) for
sets $K$ with at most $n\in \N$ components. Now let $K$ be a
closed set in $\R^2$ with $n+1$ components that separates $A$
and $B$. Then we can decompose $K$ as $K=C\cup K'$, where $C$ is
a component of $K$, and $K'$ consists of the other $n$
components of $K$.  Since $K$ is closed, the sets $K'$ and $C$
are also closed.

Let $S\sub A$ be the set of all points $a\in A$ that are
separated from $B$ by $C$. Similarly, let $S'$ be the set of all
points $a\in A$ that are separated from $B$ by
$K'$. Janiszewski's lemma now implies  that $A=S\cup S'$. Here we
are using the assumption that $B$ is a singleton set.
 
If $S=A$, then $C$ separates $A$ and $B$, and we are
done. Assuming from now on that $S\neq A$, we will show in the
following that $S'=A$. Since $K'$ has $n$ components, our
induction hypothesis then applies. This means there exists a
component of $K'$, and hence also a component of $K$, that
separates $A$ and $B$. This will finish the argument in this
case.

If $S=\emptyset$, then $S'=A$, which is our desired statement.
So we are reduced to the case where $S\ne \emptyset$ and
$S\ne A$.  

\smallskip
{\em Claim~1.} The sets $S$ and $S'$ are relatively closed in
  $A$.

 \smallskip  
Assume first that $S$ is not relatively closed in $A$. Then
there is a sequence $\{a_n\}$ of points in $S$ with $a_n\to a$
as $n\to \infty$, where $a\in A\setminus S$. Then $C$ does not
separate $a$ and $B$, and so there is a path $\ga$ in $\R^2$
joining $a$ and $B$ that does not meet $C$. In particular,
$a\notin C$, and so there exists a small path-connected
neighborhood of $a$ in $\R^2$ disjoint from $C$; but then by
traveling first from $a_n$ to $a$ in this neighborhood and then
along $\ga$, for large $n$ we can join $a_n$ and $B$ by a path
that avoids $C$, contradicting our assumption that $a_n\in S$.
Thus $S$ is relatively closed in $A$. The argument that $S'$ is
relatively closed in $A$ is completely analogous. Claim~1 is
proved.

\smallskip
{\em Claim~2.} $S\cap S' \cap C\ne \emptyset$.

 \smallskip
Recall that $S\neq \emptyset$  and $S\neq A$. Since
$S\subset A$ is relatively closed in the connected set $A$,
the set $S$ cannot be relatively open in $A$.

Hence there exists a point $a\in S$ and a sequence $\{a_n\}$ of
points in $A\setminus S$ with $a_n\to a$ as $n\to \infty$. Thus
the sequence $\{a_n\}$ is contained in $S'$. Since $S'$ is
relatively closed in $A$, it follows that $a\in S'$.

Moreover,  $a\in C$; for otherwise, we can again find a small
path-connected neighborhood of $a$ in $\R^2$ disjoint from
$C$. Then for large $n$, we could travel from $a$ to $a_n$ in
this neighborhood, and then, since $a_n\in A\setminus S$, from
$a_n$ to $B$ along a path disjoint from $C$. This contradicts
the fact that $a\in S$. So $a\in S\cap S' \cap C$ and Claim~2 follows. 


\begin{figure}
  \centering
  \begin{overpic}
    [width=10cm, tics=20,
    ]
    {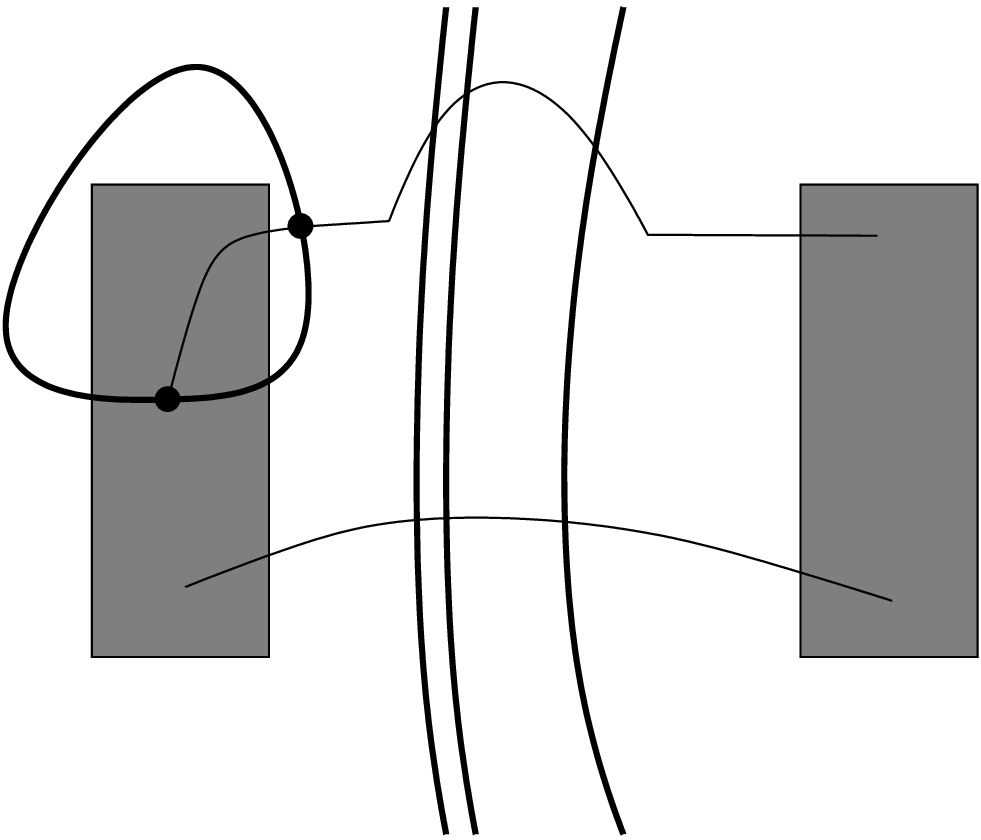}
    \put(16,14){$A$}
    \put(90,14){$B$}
    \put(16,41){$a$}
    \put(32,64){$c$}
    \put(4,71.5){$C$}
    \put(53,78){$\beta$}
    \put(19,60){$\alpha$}
    \put(50,2){$K'$}
    \put(33,34){$\gamma$}
    \put(11,60){$S$} 
  \end{overpic}
  \caption{Proof of Lemma~\ref{lem:ABKsep}.}
  \label{fig:ABK}
\end{figure}

\smallskip 
{\em Claim~3.} $K'$ separates $C$ and $B$.

\smallskip 
The ensuing argument is illustrated in Figure~\ref{fig:ABK}.
Suppose  this  claim  is not  true. Then   we can find a path $\beta$
that avoids $K'$ and joins a point $c\in C$ to $B$. Now $K'$ is
closed and so $\R^2\setminus K'$ is a union of open
regions. Since $C\sub \R^2\setminus K'$ is connected, the set
$C$ is contained in one of these regions. Since regions are
path-connected, we can find a  path $\alpha$ that joins $c\in C$ and a point $a\in S\cap S'\cap C$
(as provided by  Claim~2),   and  avoids
$K'$. Concatenating $\alpha$ and $\beta$, we obtain  a path in
$\R^2\setminus K'$ joining $a$ and $B$. This is impossible since
$a\in S'$, meaning that $K'$ separates $a$ and $B$. This
finishes the proof of Claim~3. 

\smallskip 
{\em Claim~4.} $K'$ separates $A$ and $B$.

\smallskip 
Indeed, suppose $\ga$ is a path joining $A$ and $B$.  We claim
that it meets $K'$. Since $K=K'\cup C$ separates $A$ and $B$, it
must meet $K'$ or $C$. If it meets $C$, then it also meets $K'$
as $K'$ separates $C$ and $B$ by Claim~3. So $\ga$ meets $K'$ in any case. 
 Claim~4 follows.

\smallskip
We can now apply the induction hypotheses to $K'$. Since 
$K'$ separates $A$ and $B$, and has only   $n$ components,    there exists a component of $K'$, and hence also  a component of $K$, that   separates  $A$ and $B$.

\smallskip 
{\em General Case:}  $A,B\sub \R^2$ are arbitrary 
connected sets. 

\smallskip
First note that the statement in the above special case (with
the roles of $A$ and $B$ reversed) gives the following version
of Janiszewski's lemma: Let $A$ be a singleton set in $\R^2$,
$B\sub \R^2$ be connected, and $K,L\sub \R^2$ be closed sets
with finitely many components. If $K\cap L =\emptyset$ and
$K\cup L$ separates $A$ and $B$, then $K$ or $L$ separates $A$
and $B$.

Indeed, by what we have seen, one of the components of $K\cup L$
separates $A$ and $B$,  which implies that $K$ or $L$ separates
$A$ and $B$.  

The proof in the general case is now a repetition of the proof
in the special case. The only difference is that we apply the
above version of Janiszewski's lemma instead of the original
version. We used it only once: to show that $A= S\cup S'$. 
In the general case, where  $B\subset \R^2$ is connected, but not
necessarily a singleton set,  each point $a\in A$ is separated from $B$
by $K= C \cup K'$. So by the modified version of Janiszewski's lemma
$a$ is separated from $B$ by $C$ or by $K'$. Thus $a\in S$ or $a\in S'$,
 and so again $A= S\cup S'$.
 The rest of the proof is concluded as before.

This completes  the proof if $\Om$ is homeomorphic to
$\R^2$. The proof in the case $\Om=S^2$ is the same. Here we apply the $S^2$-version of Janiszewski's
lemma mentioned after the formulation of the $\R^2$-case.
\end{proof}

\section{Orientations on surfaces}
\label{sec:orient}

Orientation is a subject that is easy to grasp on an intuitive level, but is notoriously difficult to discuss   rigorously without some sophisticated mathematical concepts or facts. 
We first  recall the fairly standard way of introducing orientation for surfaces  by using   homology groups (see \cite{Ha} for general 
  background), and then discuss an alternative and very intuitive approach to orientation based on the concept of a flag.

 Let $M$ be a compact and connected $n$-dimensional topological manifold (without boundary). If  the singular homology group $H_n(M)$ (with coefficients in $\Z$) is 
  isomorphic to $\Z$, then we call $M$ 
{\em orientable} (see \cite[Section~3.3]{Ha} for a more detailed discussion).\index{orientation} 
This is true, for example,  if 
 $M$ is a $2$-sphere or a $2$-dimensional torus (the only cases we are  interested in).
 
We say that $M$ is  {\em oriented} if one of the  two generators 
of  $H_n(M)\cong \Z$  has been chosen as  the 
{\em fundamental class}\index{fundamental class} 
$[M]$ of $M$. If    $f\: M\ra N$ is a  homeomorphism  between compact and connected 
oriented $n$-dimensional topological manifolds  $M$ and $N$,  then $f$  induces an  isomorphism  $f_*\: H_n(M)\ra H_n(N)$, and so 
$f_*([M])=[N]$ or $f_*([M])=-[N]$. In the first case we say that
$f$ is 
{\em orientation-preserving},\index{orientation!preserving}  
and in the second  that $f$ is {\em orientation-reversing}. 

In this framework we can also define
the {\em (topological) degree} of a continuous map. Namely, if $M$ and $N$ are
oriented $n$-dimensional topological manifolds with fundamental
classes $[M]$ and $[N]$, respectively, and $f\: M\ra N$ is a
continuous map, then its 
{\em degree}\index{degf@$\deg(f)$}
$\deg(f)\in \Z$ is the unique integer such that $f_*([M])=\deg(f) [N]$, where $f_*\: H_n(M)\ra H_n(N)$ is the map between homology groups induced by $f$. Note that the sign of $\deg(f)$ depends on the orientations chosen on $M$ and $N$. If $M=N$ and we choose the same orientation in source and target, then $\deg(f)$ is independent of this choice.

The degree is multiplicative in the following sense: if $f\: M\ra N$ and 
 $g\: N\ra K$ are continuous maps between oriented $n$-manifolds, then 
\begin{equation} \label{eq:degmult}
\deg(g\circ f)=\deg (g) \cdot \deg(f) . 
\end{equation} 
This immediately follows from the relation $(g\circ f)_*=g_*\circ f_*$ for the induced maps on homology (see \cite[p.~134]{Ha}).   

For open manifolds  or manifolds  with boundary one has to
resort to suitable relative homology groups to give precise
definitions for concepts related to orientation. This is
somewhat technical and we will discuss this only  in a simple
relevant case to give the general idea.    
 
Let $M$ be a {\em  surface}, i.e., a $2$-dimensional topological manifold. 
We assume that $M$ is compact, connected, and oriented.  Then the   orientation on $M$  induces an orientation on every Jordan region $X\sub M$ which in turn 
induces an orientation on $\partial X$ and  on  every arc $\alpha \sub \partial X$. 
These orientations are represented by generators in the homology groups 
$H_2(X, \partial X)$, $H_1(\partial X)$, and $H_1(\alpha, \partial \alpha)$, respectively. 

To see how to get canonical generators in these groups from    the fundamental class  of $M$, first note that we have natural isomorphisms   $$ H_2(M)\cong  H_2(M, M\setminus \inte(X)) \cong 
 H_2(X,\partial X)\cong \Z$$
induced by the inclusion map and excision (see \cite[Section~2.1]{Ha}  for the relevant terminology and facts). Hence we get an induced orientation on $X$ as represented by a generator of $H_2(X,\partial X)$ obtained as the image of $[M]$. 

Similarly, we have natural  isomorphisms $ H_2(X,\partial X)\cong H_1(\partial X)$
(from the long exact sequence of relative homology) and $H_1(\partial X)\cong
H_1(\alpha, \partial \alpha)$. They  give us canonical generators of the relevant 
homology groups once we have an orientation on $M$.   

On  a more intuitive level,  an orientation of an arc is just a
selection of one of the endpoints as the {\em initial point} and
the other endpoint as  the {\em terminal point}.  This  can
easily be reconciled with the homological viewpoint  if one uses
the 
 isomorphism 
$H_1(\alpha, \partial \alpha)\cong \widetilde H_0(\partial \alpha)$ for reduced homology. 	

The orientation of  a Jordan curve $J$ (such as $J=\partial X$) is
given by a choice of a generator in $H_1(J) \cong \Z$.   It induces a unique orientation on each arc $\alpha \sub J$ by the natural  isomorphism $H_1(J) \cong H_1(\alpha, \partial \alpha)$. 
One can think of an orientation of $J$  essentially as a direction or sense how to run through $J$ in some parametrization. It is uniquely determined  by the induced  orientation of any subarc $\alpha \sub J$. Another 
way to represent an orientation of $J$ is by a cyclic order of $k\ge 3$ points on $J$ (see Section~\ref{sec:labelings} for a related discussion).

Suppose  the Jordan region $X\sub M$ in    the oriented surface $M$ is equipped with the induced orientation.  If  $\alpha\sub \partial X$ is an arc with a given orientation, then we say that $X$ lies {\em to the left} or {\em to the right} of $\alpha$ depending 
on whether the orientation on $\alpha$ induced by the orientation of $X$ agrees with the given orientation on $\alpha$ or not. 
Similarly, we say that with a given orientation of $\partial X$ the Jordan region $X$ lies to the left or right of $\partial X$. 

Another way to introduce orientation is by using the notion of a
flag. We will outline this only  for surfaces. Let $M$ be a  connected  (possibly open) surface. By definition a {\em (topological) flag}\index{flag}    
on $M$ is a triple $(c_0,c_1, c_2)$, where $c_i\subset M$ is an $i$-dimensional cell for $i=0,1,2$ with   $c_0\sub \partial c_1$  and $c_1\sub \partial c_2$.  So a flag in $M$ is a closed Jordan region  $c_2$ with an arc $c_1$ contained   in its boundary, where the point in $c_0$ is one of the endpoints of  $c_1$. We  orient  the arc $c_1$  so that the 
point in $c_0$ is the initial point in $c_1$.  If we already have an orientation on $M$, then the  flag is called {\em positively-} or {\em negatively-oriented} (for the given orientation on $M$) depending on whether $c_2$ lies to the left or to the right of the oriented arc  $c_1$. 

This can be turned around to give an alternative definition of
orientation. 
Namely, we call two flags $(c_0,c_1, c_2)$ and  $(c'_0,c'_1, c'_2)$ in $M$ {\em equivalent} if there exists 
a homeomorphism $f\: M\ra M$ that is isotopic to $\id_M$ and satisfies 
$f(c_i)=c'_i$ for $i=0,1,2$.  On every connected surface $M$ there are at most two equivalence classes of flags. This can easily be derived from the fact that if $X,Y\sub M$ are Jordan regions, then there exists a  homeomorphism $f\: M\ra M$ that is isotopic to $\id_M$ and satisfies 
$f(X)=Y$. To get such a homeomorphism, one shrinks $X$ and $Y$ by isotopies on $M$  into small neighborhoods of points $x\in \inte(X)$ and $y\in \inte(Y)$, and moves the neighborhood of $x$ to the neighborhood of $y$ by an isotopy. In the shrinking process it is important that for every Jordan region $Z\sub M$ there is a Jordan region $Z'\sub M$ such that $Z\sub \inte(Z')$.  This
 easily follows from the fact that the topological circle $\partial Z$ is ``tame'' and so has a neighborhood that is homeomorphic to an annulus.  
 
 This outline of the argument also makes it obvious that  the homeomorphism
 $f$ on $M$ that is isotopic to $\id_M$ and satisfies  $f(X)=Y$ can be constructed so that it agrees with $\id_M$ outside a suitable compact subset of $M$.

We call $M$ {\em orientable} if there exist precisely two such equivalences classes of flags in $M$.  An {\em orientation} on $M$ is a choice of one of the 
equivalence classes as a family of distinguished flags. We say that the flags in this class are {\em 
positively-oriented} and the flags in the other class are {\em negatively-oriented}.

Any  positively-oriented  flag determines the orientation uniquely. So on orientable surfaces such as the plane $\C$ or a $2$-sphere we can think of an orientation just as a choice of some flag as positively-oriented.

 The standard orientation on $\C$ or on $\CDach$ is the one for which  the {\em standard flag} 
$(c_0, c_1, c_2)$ is positively-oriented, where 
$c_0=\{0\}$, $c_1=[0,1]\sub \R$, and 
$$c_2=\{z\in \C: 0\le \real(z)\le 1,\ 0\le 
\imag(z) \le \real(z)\}. $$

Let  $M$ be  an oriented surface,  and $\Om$ be a region in
$M$. Then $\Om$ is orientable. Essentially, this follows from the
fact that an isotopy on $\Om$ between flags in $\Om$ 
can be chosen so that it fixes points 
outside a sufficiently large 
compact subset of $\Om$. 
This allows one to extend the isotopy to $M$.  

We can represent the orientation on $M$ by a flag in $\Om$. This flag represents a unique 
orientation on $\Om$, called the {\em induced orientation} on $\Om$. 

An orientation on a not necessarily connected surface is a choice of an orientation on each of its connected components (if each of these components is orientable). If $U$ is an arbitrary open subset of an oriented surface $M$, then each of the components of $U$ is  contained in a component of $M$. We equip each of these components of $U$ with the induced orientation from the corresponding component of $M$. This defines the induced orientation on $U$.  

If $f\: M\ra N$ is a homeomorphism between connected and oriented surfaces $M$ and $N$, then either $f$ maps all positively-oriented flags in $M$ to positively-oriented flags in $N$, or   
all positively-oriented flags in $M$ to negatively-oriented flags in $N$. We say that $f$ is {\em orientation-preserving}  in the first case, and {\em orientation-reversing} in the second. 

A continuous map $f\: M\ra N$ is called a 
{\em local homeomorphism}\index{local homeomorphism}  
if each point $p\in M$ has an open neighborhood $U\sub M$ such that $f|U\: U \ra 
V\coloneqq f(U)$ is a homeomorphism of $U$ onto $V$. It follows from the 
``invariance of domain'' (see \cite[Theorem 2B.3, p.~172]{Ha}) that then $V$ is  an open subset of $N$. By shrinking $U$ if necessary,   we can always assume here that $U$ and $V$ are topological disks. If, in addition, 
these homeomorphisms $f|U$ preserve orientation, then we call
$f$ 
{\em orientation-preserving}.\index{orientation!preserving!local homeomorphism}\index{local homeomorphism!orientation-preserving} 
Roughly speaking, this means that $f$ sends a small  positively-oriented flag near a point in $M$ to  
a positively-oriented flag in $N$. 

\begin{lemma}\label{lem:23orient} Let $M$, $N$, $K$ be 
 connected and oriented  surfaces, and  
   $f\colon M\to K$, $g\colon N\to K$, and 
  $h\colon M\to N$ be local homeomorphisms  such that $f=g\circ h$. 
  If two of the maps $f$, $g$,  $h$ are orientation-preserving, then  
the third map is orientation-preserving as well. 
\end{lemma}

\begin{proof} We have to consider three cases depending on which two 
of the maps $f$, $g$,  $h$ are orientation-preserving. We will only consider the case when $f$ and $h$  are orientation-preserving, and show that then $g$ has the same property. The other two cases are very similar and we leave the details to the reader.

Since $N$ is connected, $g$ either preserves orientation near all  points in $N$, or reverses it. In order to decide this, it suffices to consider $g$ near one point $q\in N$ and verify that $g$ preserves the orientation of one positively-oriented flag $F$ contained in a small topological disk $V$ with $q\in V$ such that the map $g|V$ is a homeomorphism
onto its image. 

 We may assume that $q\in h(M)$. 
Then there exists $p\in M$ with
 $h(p)=q$, and we can find a topological disk $U\sub M$ with $p\in U$  
 such that $f|U$ and $h|U$ are orientation-preserving
 homeomorphisms onto their images. By replacing $V$ with  a
 smaller topological disk if necessary, we may assume that
 $h(U)=V$ and that there is a flag $F'\sub U$ with $h(F')=F$
 (here and below we use the obvious definition for image flags
 such as $h(F')$). 
Since $h|U$ is orientation-preserving and $F$ is positively-oriented, the flag $F'$ must also be 
  positively-oriented. Since $f|U$ is orientation-preserving, the image flag  $f(F')=g(h(F'))=g(F)$ of $F'$ under $f$, which agrees with the image of $F$ under $g$,  is  positively-oriented. Hence $g$ is orientation-preserving.
 \end{proof}

   \section{Covering maps}
\label{sec:covmaps} 

Before we turn to {\em branched} covering maps, we remind the reader of some well-known facts about covering maps. They are true in great generality, but we restrict ourselves mostly to covering maps between surfaces  (see 
\cite[Section~1.3]{Ha} and \cite[Chapter~1]{Fo81} for a more detailed discussion).
Here and also in the following section  a 
{\em surface}\index{surface} 
is a connected and orientable 
$2$-dimensional topological manifold. So in contrast to Section~\ref{sec:orient}  we use this term in a more restrictive sense. We assume that a surface is oriented by specifying an equivalence class of positively-oriented flags (as discussed in 
Section~\ref{sec:orient}). 

Let $X$ and $Y$ be (oriented) surfaces, and $\pi\: X\ra Y$ be a continuous and surjective map. Then 
$\pi$ is called a 
{\em covering map}\index{covering map}
if   every point $p\in Y$ has  an open  and connected 
 neighborhood $V\sub Y$ such that $\pi^{-1}(V)$ can be written as a disjoint union
 $$\pi^{-1}(V)=\bigcup_{i\in I} U_i$$ of open and connected
 sets  $U_i\sub X$ such that $\pi|U_i$ is an
 orientation-preserving homeomorphism of $U_i$ onto $V$  for
 each $i\in I$. Here $I$ is some index set. We say that a set
 $V$ as in this definition is 
 {\em evenly covered}\index{evenly covered} 
 by $\pi$. By possibly shrinking the set,
 one can always assume that $V$ is a topological disk. 
 
 Usually, one does not insist on the maps $\pi|U_i$  being orientation-preserving; this additional requirement is motivated by our definition of a branched covering map
  in the next section: without it a covering map would not necessarily be a branched covering map.  

 If $\pi\: X\ra Y$  is a covering map and we want to emphasize
 $Y$, then we say that $\pi$ is a 
{\em covering map over} $Y$.\index{covering map!over $Y$} 
 The covering map  $\pi\: X\ra Y$ is called {\em finite} if $\pi$ is {\em finite-to-one} in the sense that every point in $q\in Y$ has only finitely many preimages in $X$. In this case, the cardinality $\#\pi^{-1}(q)$ is constant and independent of $q\in Y$.

A covering map  is an orientation-preserving   local homeomorphism;   so for  every point $x\in X$ there exists  an
 open neighborhood $U$ such that $\pi|U$  is an orientation-preserving   homeomorphism of $U$ onto $\pi(U)$. 
 Conversely, if $X$ and $Y$ are compact, then   every orientation-preserving local  homeomorphism $\pi\: X\ra Y$ is a covering
  map. 
  
  If $\pi \:X \ra Y$ is a covering map, then a homeomorphism 
  $g\: X\ra X$ is called a 
  {\em deck transformation}\index{deck transformation} 
  of $\pi$ if $\pi=\pi\circ g$. These maps $g$ form a group $G$
  called the {\em deck transformation group} of $\pi$. If $X$ is
  simply connected, then $G$ is isomorphic to the fundamental
  group of $Y$ (see the discussion below and 
\cite[Theorem~5.6]{Fo81}). 
  
  Let $\pi\: X\ra Y$ be a covering map, $Z$ a topological space, and $f\: Z\ra Y$ be a continuous map.
  A continuous map $g\: Z\ra X$ is called a 
{\em  lift}\index{lift!of map}
 of $f$ (by $\pi$) if $\pi\circ g=f$. In this case, we have the commutative diagram:
  \begin{equation}\label{eq:basicliftdia}
  \xymatrix {
  & X  \ar[d]^{\pi}   \\
    Z  \ar[ur]^{g}  \ar[r]^{f}  &   Y\rlap{.}
  }
\end{equation}

The next lemma is a standard fact about existence and uniqueness of lifts
(for the terminology and the proofs see  
\cite[Section~1.3, Proposition~1.34, and Proposition~1.33]{Ha}; 
see also
\cite[Section~1.4 and Theorem~4.17] {Fo81}).

\begin{lemma}[Existence and uniqueness of lifts] \label{lem:liftsofcov}
Let $X$ and $Y$ be (oriented) surfaces,  $\pi\: X\ra Y$ be a covering map, and $Z$ be a 
path-connected and locally path-connected topological space.

 \begin{enumerate}
 
   \item
    \label{item:liuniq}
   Suppose $g_1, g_2\: Z\ra X$ 
   are two continuous  maps such that $\pi\circ g_1=\pi\circ g_2$. If there exists $z_0\in Z$ with  
 $g_1(z_0)=g_2(z_0)$, then $g_1=g_2$.  
   
  \item 
    \label{item:Liex}
    Suppose  $Z$ is simply connected,  $f\: Z\ra Y$ is a  continuous map, and  $z_0\in Z$ and $x_0\in X$ are points such that $f(z_0)=\pi(x_0)$. Then there exists a  continuous  map
    $g\: Z\ra X$ such that $g(z_0)=x_0$ and  $f=\pi\circ g$.

     \end{enumerate}
\end{lemma}  

In \ref{item:liuniq} the maps $g_1$ and $g_2$ are lifts of $f\coloneqq \pi\circ g_1=\pi\circ g_2$. 
So the statement says that lifts of maps are uniquely determined by the image of one point. 

Statement \ref{item:Liex} guarantees the existence of a lift $g$  of $f$ with $g(z_0)=x_0$. By \ref{item:liuniq} this lift $g$ of $f$ satisfying $g(z_0)=x_0$ is unique. 

A special and important case 
is if $Z=[0,1]$, and $z_0=0$. Then $f$ is a path in $Y$,  and the statement says that 
we can lift it to a unique path $g$ in $X$ if we prescribe any point in  the fiber $\pi^{-1}(f(0))$ as the initial point of the lift.

Let $X$ be a path-connected and locally path-connected topological space, and $x_0\in X$ be a basepoint in $X$.
Then   the {\em fundamental group} $\pi_1(X, x_0)$ of $X$ with respect to $x_0$ consists of all homotopy classes of  loops in $X$ based  (i.e., starting and ending) at $x_0$  (see \cite[Section~1.3]{Ha} for precise definitions). Simple connectivity of $X$ means that   
$\pi_1(X, x_0)$ is the trivial group only consisting of the unit element. 

Suppose $Y$ is another path-connected and locally path-connected topological space with basepoint 
$y_0$, and $f\: X\ra Y$ a continuous map that is 
basepoint-preserving in the sense that $f(x_0)=y_0$. If we  assign to each class $[\ga]\in \pi_1(X, x_0)$ represented by a loop $\ga$ in $X$ based at $x_0$, the class $[f\circ \ga]$ represented 
by the image loop $f\circ \ga$, then we get a well-defined induced group homomorphism 
$f_*\: \pi_1(X,x_0)\ra\pi_1(Y,y_0)$. 

Suppose $X$ and $Y$ are surfaces, $\pi\: X\ra Y$ is a covering map, and $X$ is simply connected. Then $\pi\: X\ra Y$
is a 
{\em universal covering map}:\index{covering map!universal} 
if $f\: Z\ra Y$ is another covering map from a surface $Z$, $x_0\in X$ and $z_0\in Z$ with 
$\pi(x_0)=f(z_0)$, then  there exists a covering map $g\: X\ra Z$ such that $\pi=  f \circ g$ and $g(x_0)=z_0$. A universal 
covering map $\pi\: X\ra Y$ exists for each surface  and  is unique up to equivalence: 
if $\widetilde \pi\: \widetilde X\ra Y$ is another universal 
covering map, then there exists a homeomorphism $\varphi\: X\ra \widetilde X$ such that $\pi =\widetilde \pi\circ \varphi$
(see \cite[Section~1.5]{Fo81}).

    \section{Branched covering maps}
\label{sec:appbracovmap} 

 In this section we discuss branched covering maps between surfaces.  Since it is difficult to find references for
this topic in the literature, our exposition is rather detailed and we provide 
 proofs for the statements discussed. We will sometimes skip details
 if they are straightforward to fill in.  
 
  We will 
first define the relevant terminology; in particular, we will give a precise 
definition of a \emph{branched covering map} and what it means for a neighborhood of a point to be \emph{evenly covered} in this context. 
Useful criteria for verifying the relevant conditions are   provided by Lemmas~\ref{lem:z^d} and~\ref{lem:evennei}.

Lemma~\ref{lem:pbackcostr} shows that a
conformal structure can be pulled back by a  branched covering map.
This implies that one can often reduce to the holomorphic case 
if one studies such  maps. In particular, every 
 branched covering map
on a $2$-sphere can be represented by a rational map on the Riemann sphere  up to suitable homeomorphic coordinate changes 
in  source and target (see Corollary~\ref{cor:brcovratup}).

The main difficulty in the proof 
of Lemma~\ref{lem:pbackcostr} is the behavior of the given map 
near  branch points; this is resolved by what can be 
viewed as a variant of Riemann's removability
theorem. 

 Lemma~\ref{lem:2_3_branched} is another useful criterion  if one wants to check whe\-ther a map is a branched covering map. It 
 essentially says that 
if three  continuous maps
$f$, $g$, and $h$ between  surfaces 
satisfy $f=g\circ h$, and if two of the maps are branched covering maps, then the third one is a  branched covering map as well.  A similar statement is true for 
holomorphicity of the maps $f$, $g$, and $h$. 

In the last part of this section we consider existence and
uniqueness statements for lifts by  branched covering maps (see
Lemma~\ref{lem:liftsofpathsbranched} and
Lemma~\ref{lem:liftsofbranched}).

As in the previous section, we again make the standing assumption
that each surface is connected and oriented.  Let $X$ and $Y$ be
compact surfaces, and $f \: X\ra Y$ be a continuous and
surjective map. Recall from Section~\ref{sec:branched-coverings}
that $f$ is a 
{\em branched covering map}\index{branched covering map}\index{map!branched covering} 
if for each point $p\in X$ there exists
$d\in \N$, topological disks $U\subset X$ and $V\subset Y$ with
$p\in U$, $q\coloneqq f(p)\in V$, and orientation-preserving
ho\-meo\-mor\-phisms $\varphi\: U \ra \D$ and $\psi\: V\ra \D$
with $\varphi(p)=0$ and $\psi(q)=0$ such that
\begin{equation}\label{eq:brzd}
  (\psi\circ f \circ \varphi^{-1})(z)=z^d
\end{equation}
for all $z\in \D$.  
 
So branched covering maps are modeled on non-constant holomorphic
maps between compact Riemann surfaces. Every such map is a
branched covering map.

For maps between surfaces that are not necessarily compact, one
has to adjust the definition of a branched covering map. Recall
that for (unbranched) covering maps we require that each point in
the target has a neighborhood that is evenly covered. For
branched covering maps we impose a similar condition. It is always true 
for compact surfaces (see the discussion
after the proof of Lemma~\ref{lem:evennei}). Accordingly, we make
the following definition.

\begin{definition}[Branched covering maps]
  \label{def:brcovmap}
  \index{branched covering map|textbf}
  \index{map!branched covering|textbf}
  Let $X$ and $Y$ be (connected and oriented) surfaces, and
  $f\: X\ra Y$ be a continuous map.  Then $f $ is a {\em branched
    covering map} if for each point $q\in Y$ there exists a
     topological disk  $V\subset Y$ with $q\in V$ that is {\em evenly
    covered}\index{evenly covered} by $f$ in the following sense:
  for some index set $I\ne \emptyset$ we can write $f^{-1}(V)$ as
  a disjoint union
 $$f^{-1}(V)=\bigcup_{i\in I} U_i$$ of open  sets  $U_i\sub X$ such that $U_i$ 
 contains precisely one point $p_i\in f^{-1}(q)$. Moreover, we require that for each $i\in I$  there 
 exists $d_i\in \N$, 
  and orientation-preserving homeomorphisms
 $\varphi_i\: U_i \ra \D$ and $\psi_i\: V\ra \D$ with $\varphi_i(p_i)=0$ and $\psi_i(q)=0$ such that 
 \begin{equation}\label{eq:z^dbrcov}
  (\psi_i\circ f \circ \varphi_i^{-1})(z)=z^{d_i}
  \end{equation}
 for all $z\in \D$. 
 \end{definition}
  Note that the sets  $U_i$, $i\in I$, are the connected components of $f^{-1}(V)$, and that each $U_i$ is a also 
  a topological disk.    Moreover, \eqref{eq:z^dbrcov} implies that $f(U_i)=V$ for  $i\in I$. 
 
 For
 given $f$ the number $d_i$ is uniquely determined by $p=p_i$ and
 called the 
 {\em local degree}\index{local degree|textbf}\index{deg@$\deg_f(p), \deg(f,z)$|textbf} 
 of $f$ at $p$, denoted by
 $\deg_f(p)$ or $\deg(f,p)$. Our definition allows different local degrees
 at points in the same fiber $f^{-1}(q)$. Note that  if $q'$ is a point close to, but distinct from $q= f(p)$, then 
  $\deg(f,p)$ is equal 
 to the number of distinct preimages of $q'$ under $f$ close to $p$.   In particular, 
 near $p$ the map $f$ is $d$-to-1, where $d= \deg(f,p)$.

 Every branched covering map $f\: X\ra Y$ is surjective, 
 {\em open}\index{open map}\index{map!open} 
(images of open sets are open), and 
{\em discrete}\index{discrete}\index{map!discrete}
 (the preimage set  of every point is {\em discrete in} $X$, i.e., it has no limit points in $X$). Every covering map is also a branched covering map.

 A 
 \emph{critical point}\index{critical!point} 
 of a branched covering map $f\colon X\to Y$
 is a point $p\in X$ with $\deg_f(p)\ge 2$. A 
 {\em critical value}\index{critical!value} 
 is a point $q\in Y$ such that the fiber $f^{-1}(q)$ contains a critical point of $f$. 
The set of critical points of $f$ is discrete in $X$; indeed,  if $p\in X$ is arbitrary, then there exists an open neighborhood $U$ of $p$ such that $f$ is a local homeomorphism on $U\setminus \{p\}$. So the set of critical points of $f$ cannot have a limit point in $X$. Similarly,  the set of critical values of $f$ is discrete in $Y$, because 
  if $V\subset X$ is an evenly covered neighborhood of a point $q\in Y$,
 then $q$ is the only possible critical value of $f$ in $V$. 
If 
$f\colon X\to Y$ is a branched 
covering map, then $f$ is an orientation-preserving local homeomorphism
near each point $p\in X$ that is not a critical point of $f$.

 A continuous map $f\: X\ra Y$ between surfaces is called 
 {\em proper}\index{proper!map}\index{map!proper} 
if $f^{-1}(K)$ is compact for every compact set $K\sub Y$. We record the following useful fact.

\begin{lemma} \label{lem:proper} Let $f\: X\ra Y$ be an open and continuous map between surfaces $X$ and $Y$.

\begin{enumerate}
 
   \item
    \label{item:proper1}
 If $f$ is proper, then $f(X)=Y$.    
    \smallskip
  \item 
    \label{item:proper2}
 Suppose $V\sub Y$ is a region  and  $U$   a connected component of $f^{-1}(V)$. If $f$ is proper or if $\overline U$ is compact, then $f|U\: U \ra V$ is a proper map,  $f(U)=V$, 
 and $f(\partial U) \sub \partial V$.          
   \end{enumerate}
\end{lemma} 
In particular, a proper, open, and continuous map $f\: X\ra Y$ between surfaces is surjective.

\begin{proof} \ref{item:proper1} It follows from  our hypotheses that $f(X)$ is a non-empty open set and  from our definition of a surface that $Y$ is connected. So it suffices to show that $f(X)$ is closed. To see this, 
 let $\{y_n\}$ be a
sequence in $f(X)$ and suppose that $y_n\to y\in Y$ as $n\to \infty$.
Then for each $n\in \N$ there exists $x_n\in X$ with $f(x_n)=y_n$. 

Now the set 
$K\coloneqq \{y\} \cup\{y_n:n\in \N\} \sub Y$ is compact. Since $f$ is proper, 
the set $f^{-1}(K)\sub X$ is also compact. Since $\{ x_n\}$ is a sequence in $f^{-1}(K)$,
it has a convergent subsequence. By passing to a subsequence if necessary, we may assume that $\{x_n\}$  itself converges, say $x_n\to x\in X$ as $n\to \infty$. 
Then by continuity of $f$ we have 
$$ y=\lim_{n\to \infty} y_n = \lim_{n\to \infty} f(x_n)=f(x). $$ 
Hence $y\in f(X)$ and so $f(X)$ is indeed closed.

\smallskip 
 \ref{item:proper2} Let $K\sub V$ be compact. To see  that $f|U$ is proper, 
 we have to show that $(f|U)^{-1} (K)=U\cap f^{-1}(K) $ is
 compact. To this end, let 
 $\{x_n\}$ be an arbitrary sequence in $U\cap f^{-1}(K)$. If $f$  is proper, then  $f^{-1}(K)$ is compact, and so  $\{x_n\}$ has a  convergent 
 subsequence. This is also true if $\overline U$ is compact.
By passing to a subsequence, we may assume that $\{x_n\}$ itself converges, say $x_n\to x \in \overline U$. We have to show that actually  $x\in U\cap f^{-1}(K)$.
 
 By continuity of $f$ the point $f(x)$ is the limit of the sequence 
 $\{f(x_n)\}$ which lies in $K$. Hence $f(x)\in K$. 
 So if $x\in U$, then $x\in U\cap f^{-1}(K)$ as desired. 
 
 The other alternative, $x\in \partial U$  is impossible. Indeed, since $f(x)\in K\sub V$,
 there exists a small connected neighborhood $N$ of $x$ with $f(N)\sub V$. Since 
 $x\in \partial U$, the set $N$ meets $U$ and so $N\cup U$ is a
 connected subset of  $f^{-1}(V)$. Since $U$ is a connected
 component of $f^{-1}(V)$, this implies $N\sub U$; but then $x$
 would be an 
interior point and not a boundary point of $U$.  
  
 The set  $U$ is also a region which implies that $f|U\: U \ra V$ is an open and continuous map between the surfaces $U$ and $V$. Since $f|U$ is also proper by what we have just seen, it follows from \ref{item:proper1}  that $f(U)=V$.
 
Finally, if $x\in \partial U$, then $f(x) \in \overline V$ by continuity of $f$.  The argument above shows that $f(x)\in V$ is impossible, and so $f(x)\in \partial V$. Hence $f(\partial U)\sub \partial V$ as desired.  
 \end{proof} 
 
 Suppose  a topological disk $V\sub Y$ is evenly covered by a continuous map 
 $f\: X\ra Y$ and $U_i$ is a component of $f^{-1}(V)$ as in Definition~\ref{def:brcovmap}. Then the map $f|U_i\: U_i\ra V$ is proper, because up to homeomorphic changes in source and target the map is given by a power map $z\mapsto z^d$ which is a proper map on $\D$.

The following statement provides a  convenient  criterion that allows us to verify the conditions  in  Definition~\ref{def:brcovmap}.
 
 \begin{lemma}\label{lem:z^d}  Let $U$ be a surface, $V$ be a topological disk, $p\in U$, $q\in V$, and $f\: U\ra V$ be a proper and continuous map. Suppose that $f^{-1}(q)=\{p\}$ 
 and  that $f$ is a local homeomorphism near each point in $U\setminus \{p\}$. 
 
 Then $U$ is also a topological disk, and 
 for each   homeomorphism  $\psi\: V\ra \D$ with 
 $\psi(q)=0$, there exists  a  homeomorphism $\varphi\:U\ra \D$ with 
 $\varphi(p)=0$ such that 
 \begin{equation}   \label{eq:z^dnonorrep}
 (\psi\circ f \circ \varphi^{-1})(z)=z^d
 \end{equation} 
 for all $z\in \D$, where $d\in \N$. 
 
 If, in addition, $\psi$ is orientation-preserving and $f$ is orientation-preserving near each  point in $U\setminus\{p\}$, then $\varphi$ is orientation-preserving as well.  
 \end{lemma}

\begin{proof} In this  proof it is convenient to  adopt a more general notion of a covering map, where we allow arbitrary, not necessarily orientation-preserving, homeo\-morphisms on components of preimages of evenly covered neighborhoods.

  Our assumptions imply 
that the restriction $f|U\setminus \{p\} $ is a proper, open, and continuous map of $U\setminus \{p\}$ into $V\setminus \{q\}$. Hence it is surjective by Lemma~\ref{lem:proper}~\ref{item:proper1}.
To see that this restriction is a covering map of  $U\setminus \{p\}$ onto $V\setminus \{q\}$, let $y\in V\setminus \{q\}$ be arbitrary. Since $f$ is proper, the set $f^{-1}(y) \sub  U\setminus \{p\}$ is  finite.

 Let $W\sub V\setminus \{q\}$ be a topological disk with $y\in W$. If $W'\sub U\setminus \{p\}$ is a connected component of $f^{-1}(W)$, then $f(W')=W$ by Lemma~\ref{lem:proper}~\ref{item:proper2}. In particular, $W'$ contains a point in 
$x\in f^{-1}(y)$. This implies that there can only be finitely many of these components $W'$ of $f^{-1}(W)$. By choosing $W$ sufficiently small, we can ensure that each component $W'$ contains precisely
one point $x\in f^{-1}(y)$. Moreover, since $f$ is a local homeomorphism near $x$, we may assume that $W'$ is a small enough neighborhood of $x$ such that $f|W'$ is a homeomorphism of $W'$ onto $W$ (this can easily be justified by an argument as for  Lemma~\ref{lem:continv}). This implies that $W$ is evenly covered by  $f|U\setminus \{p\}$ with finitely many components of $f^{-1}(W)$. Hence $f\:U\setminus \{p\}\ra V\setminus \{q\}$ is a finite covering map. 
  
 If $\psi$ is a homeomorphism as in the statement, then $\psi\circ f$ is a finite  covering map from 
 $U\setminus\{p\}$ onto $\D\setminus\{0\}$.
 Now it is a standard fact that up to  equivalence each  covering map  onto $\D\setminus\{0\}$ with finite fibers is a power map 
 $P_d(z)\coloneqq z^d$ 
 on $\D\setminus\{0\}$ for some $d\in \N$ (essentially, this is proved in 
 \cite[Theorem~5.10]{Fo81}). 
Here this means 
 that there exists a homeomorphism $\varphi\: U\setminus\{0\}\ra 
\D\setminus\{0\}$ such that $\psi \circ f = P_d\circ \varphi$ on $U\setminus \{p\}$. 
This equation implies that we get a homeomorphic extension $\varphi\: U \ra \D$ by setting $\varphi(p)=0$. Hence $U$ is a topological disk.
The first part of the statement follows.

Suppose, in addition,  that $\psi$ and  $f$ are orientation-preserving. Since $P_d$ is orientation-preserving on $\D\setminus \{0\}$, the relation  
 $\psi \circ f = P_d\circ \varphi$ on $U\setminus \{p\}$ implies that $\varphi$ must have this property as well (this follows from Lemma~\ref{lem:23orient}). 
\end{proof} 

\begin{lemma}\label{lem:evennei}
  Let  $X$ and $Y$ be surfaces, and $f\: X\ra Y$ be  an open and
  continuous map. Suppose $q\in Y$ and $V\subset Y$ is a
  topological disk  that is an  evenly covered neighborhood of
  $q$  as in  
Definition~\ref{def:brcovmap}. If $\widetilde V$ is a topological disk with 
$q\in \widetilde V\sub V$, then $\widetilde V$ is also evenly covered by $f$.
\end{lemma}

\begin{proof} Suppose we have a decomposition 
$$f^{-1}(V)=\bigcup_{i\in I}  U_i$$ into connected components
as in Definition~\ref{def:brcovmap}. Let  $\widetilde U\subset X$ be  a connected component 
of $f^{-1}(\widetilde V)$. Then there exists a unique $i\in I$ such that
 $\widetilde U\sub U_i$. We know that the map  $f|U_i\:U_i\ra V$ is proper, open, and continuous. Moreover, $\widetilde U$ is a  
component of $(f|U_i)^{-1}(\widetilde V)=U_i\cap f^{-1}(\widetilde V)$.  So by  Lemma~\ref{lem:proper}~\ref{item:proper2} the map $f|\widetilde U\: \widetilde U\ra \widetilde V$ is proper  and  we have  $f(\widetilde U)=\widetilde V$. This,  together 
  with $\widetilde U\sub U_i$, implies that $(f|\widetilde U)^{-1}(q)=\{p\}$, where 
  $p$ is the unique point in $U_i$ with $f(p)=q$. 
  
  Finally, since $V$ is evenly covered by $f$, the map  $f$ is an orientation-preserving homeomorphism near each point in 
  $\widetilde U\setminus\{q\}\sub U_i\setminus \{q\}$.
  It now follows from Lemma~\ref{lem:z^d} that up to orientation-preserving ho\-meo\-mor\-phic changes in source and target, the map
  $f|\widetilde U\: \widetilde U \ra \widetilde V$ can be represented by a power map 
  $z\mapsto z^d$, where $d\in \N$. Since this is true for each component  $\widetilde U$ of $f^{-1}(\widetilde V)$, the topological disk  $\widetilde V$ is evenly covered by $f$.
\end{proof}

The arguments in the previous lemma imply that the  Definition~\ref{def:brcovmap} of  a branched covering map $f\: X\ra Y$ is equivalent to  the definition given in Section~\ref{sec:branched-coverings} in case the surfaces $X$ and $Y$ are compact.
Indeed, if in this case $f\: X\ra Y$ is a branched covering map according to the definition given in Section~\ref{sec:branched-coverings}, then each point
$q\in Y$ has finitely many distinct 
preimages $p_1, \dots, p_n\in X$ under $f$. For each point $p_i$
there exist topological disks $U_i\subset X$ and $V_i\subset Y$ with $p_i\in U_i$ and 
$q\in V_i$ such that $f|U_i\: U_i\ra V_i$ can be represented by a power map $z\mapsto z^{d_i}$ with $d_i\in \N$ up to orientation-preserving   
homeomorphic changes in source and target. We can choose a 
topological disk $V\sub V_1\cap \dots \cap V_n$ with $q\in V$. Then 
it easily follows from the arguments in the proof of Lemma~\ref{lem:evennei}
that $V$ is a neighborhood of $q$ that is evenly  covered by the map $f$. 

Away from its critical values a branched covering map is actually a covering map. This is made precise in the following statement. 

\begin{lemma}\label{lem:brcovcov}
Let $X$ and $Y$ be surfaces, and $f\: X \ra Y$ be a branched covering map. Suppose $P\sub Y$ is a set with 
$f(\crit(f))\sub P$ that is discrete in $Y$. Then 
$f\: X\setminus f^{-1}(P) \ra Y \setminus P$ is a covering map. 
\end{lemma} 
Here and in the following, for simplicity we do not distinguish
in our notation between the original map $f$ and its restriction $f|X\setminus f^{-1}(P)$.

 \begin{proof} As  a branched covering map, the map  $f$ 
 is discrete. This implies that the preimage $f^{-1}(P)$ of  $P$ is discrete in $X$. In particular, 
$ X\setminus f^{-1}(P)$  and $Y \setminus P$ are connected and
hence surfaces (equipped 
with the orientations induced 
by $X$ and $Y$, respectively). 
Moreover, if $x\in X\setminus f^{-1}(P)$, then $f(x)\in Y\setminus P$. So we can consider $f$ 
as map between the surfaces  $X\setminus f^{-1}(P)$ and   $Y \setminus P$.

  Let $q\in Y \setminus P$ be arbitrary. Then there exists a topological disk $W\sub Y$ with $q\in W$ that is evenly covered by $f$ (as in the definition of a branched covering map). Since $ q\in  Y \setminus P$ and $P$ is discrete in $Y$, we can find a smaller topological 
 disk $\widetilde W\sub W$ with $q\in  \widetilde W \sub 
 (Y\setminus P)\cap W$. Then $\widetilde W$ is also 
 evenly covered by $f$ according to 
 Lemma~\ref{lem:evennei}. 
 
 Since $f(\crit(f))\sub P$, we have $\crit(f) \sub f^{-1}(P)$ and so no component $U$ of $f^{-1}(\widetilde W)\sub X\setminus f^{-1}(P)$ contains a critical point of $f$. This implies that $f$ is an orientation-preserving homeomorphism of $U$ onto $\widetilde W$. It follows that the neighborhood $ \widetilde W$ of $q$ is  evenly covered 
 by the map  $f|X\setminus f^{-1}(P)$ (as in the
 definition of a covering map). The statement follows.
 \end{proof}

Questions  about branched covering maps can often be reduced  to the holomorphic case due to the following statement. 

\begin{lemma}\label{lem:pbackcostr}
 Let $X$ and $Y$ be surfaces, and $f\: X \ra Y$ be a branched covering map. Then for each conformal structure on $Y$ there exists a conformal structure on $X$ such that with these conformal structures on $X$ and $Y$  the map $f\: X\ra Y$ is holomorphic. 
\end{lemma}

For the proof we will use some standard facts and terminology
from the theory of Riemann surfaces. We will follow
\cite[Section~1.1]{Fo81}.  
A {\em conformal structure} on a surface is represented by a {\em complex atlas}  of {\em holomorphically compatible (complex) charts}. Every (orientable) surface admits a complex  atlas and hence a conformal structure. 
A surface equipped with such a conformal structure is called a {\em Riemann surface}. 
\index{Riemann surface}

\begin{proof}  Let $\mathcal{A}$ be  a complex  atlas representing the given conformal structure on $Y$. 

Let $E\sub X$ be the set of points where $f$ is not a local homeomorphism.
This set is discrete in $X$. Near each point $X\setminus E$ the map $f$ is an orientation-preserving local homeomorphism. If $p\in X\setminus E$, then we obtain a chart defined near $p$ by composing $f$ restricted to a sufficiently small neighborhood of $p$ with  a chart in $\mathcal{A}$
defined near $f(p)$. These charts form an atlas $\mathcal{A}'$ on $X\setminus E$. The charts in $\mathcal{A}'$ are holomorphically compatible, because the charts in $\mathcal{A}$ are. The atlas $\mathcal{A}'$ defines a conformal structure on $X\setminus E$ such that 
the restriction $f|(X\setminus E)\: X\setminus E\ra Y$ is holomorphic.

It remains to find suitable charts defined near the points in $E$. 
Let $p\in E$ be arbitrary, and $q=f(p)$. Since $f$ is a branched covering map, 
we can find small topological disks $U\subset X$ and $V\subset Y$
with $p\in U$, $U\cap E=\{p\}$, and  $q\in V$ such that the map
$f\: U \setminus \{p\} \ra V\setminus \{q\}$ is a covering map
with finite fibers (up to orientation-preserving {\em
  homeomorphic}  
changes in source and target 
it is represented by a map 
of the form  $z\mapsto z^d$ on $\D\setminus\{0\}$ with $d\in \N$).

We may assume that $V$ is so small that with the given conformal structure on $Y$ the topological disk $V$ is conformally equivalent to $\D$ and hence the punctured disk $V\setminus\{q\}$ is conformally equivalent to  
$\D \setminus \{0\}$. Since we only have a  conformal structure defined 
on $X\setminus E\supset U\setminus \{p\}$, it is a priori not clear that 
$U\setminus \{p\}$ is also conformally equivalent to $\D \setminus \{0\}$;  
but  $U\setminus \{p\}$ is  a topological annulus and hence with the given conformal structure is  conformally equivalent to a Euclidean annulus of the form 
 $$A=\{z\in \C: r<\abs{z}<R\}, $$
 where $0\le r<R\le \infty$. Then from the finite covering map $f\: U \setminus \{p\} \ra V\setminus \{q\}$ we obtain by  a conformal change a finite holomorphic covering map $g\: A \ra \D \setminus \{0\}$ that satisfies $\abs{g(z)}\to 0$ as $\abs{z}\to r$ and $\abs{g(z)}\to 1$ as $\abs{z}\to R$. Here  we have $R<\infty$, because otherwise $\infty \in \CDach$ would be a removable 
 singularity for $g$ which is easily seen to be impossible. 
 
 We also have  $r=0$. This  
  follows from the fact that the {\em conformal modulus} 
(see \cite[Appendix~B]{Mi})
 $$\text{mod}(A)=\frac {1}{2\pi}  \log (R/r)\in (0,\infty]$$ of the annulus $A$ only changes by a finite multiplicative constant under a finite holomorphic covering  map. 
 
 We may assume that $R=1$. In particular, there exists a conformal map  $\varphi \: U \setminus \{p\}\ra
  \D \setminus \{0\}$ such that $\varphi(u)\to 0$ as $u\to p$. If we extend this map by setting 
  $\varphi(p)=0$, then $\varphi$ is an orientation-preserving  homeomorphism of $U$ onto $\D$. Moreover, this map gives  a chart on $X$ defined near $p$. It   is holomorphically compatible 
  with the charts in $\mathcal{A}'$, because the map $\varphi|U \setminus \{p\} $ is holomorphic. 
  
  If we add these charts $\varphi$ defined near points $p\in E$, then we obtain an atlas $\mathcal{A}''$  of holomorphically compatible charts on $ X$. This defines a conformal structure on $X$ and it is clear that 
  with the conformal structures represented by $\mathcal{A}''$ on $X$ and $\mathcal{A}$ on $Y$, the map $f\: X\ra Y$ is holomorphic. 
\end{proof} 

The previous statement can be applied to branched covering maps on a $2$-sphere $S^2$ and gives the following result. 

\begin{cor} \label{cor:brcovratup} 
  Suppose $f\:S^2\ra S^2$ is a
  branched covering map, and $\psi\: S^2\ra \CDach$ is an
  orientation-preserving homeomorphism. 
Then there 
exists an orientation-preserving homeomorphism $\varphi \: S^2 \ra \CDach$ such that $R\coloneqq \psi\circ f \circ \varphi^{-1}$ is a rational map on 
$\CDach$. 
\end{cor}

In particular, up to homeomorphic changes  in source and target, every branched covering map on a $2$-sphere $S^2$ can be represented by a rational map on the Riemann sphere $\CDach$.

\begin{proof} The given homeomorphism $\psi\: S^2\ra \CDach$ gives a natural conformal structure on $S^2$. It is represented by an atlas obtained by pulling back charts  in an atlas representing the conformal structure on $\CDach$. If we equip $S^2$ with this conformal structure, then $\psi\: S^2\ra \CDach$ is a biholomorphism. 

 By Lemma~\ref{lem:pbackcostr} there exists another  conformal structure on $S^2$ such that the map $f\: S^2\ra S^2$ is holomorphic with respect to these two conformal structures on source and target. By the uniformization theorem the sphere $S^2$ equipped with the 
source conformal structure is conformally  equivalent to $\CDach$ by a biholomorphism  $ \varphi \: S^2\ra \CDach$. In particular, $\varphi\: S^2\ra \CDach$ is a homeomorphism and  
$R =\psi \circ f \circ  \varphi^{-1}$ is a holomorphic map on $\CDach$. So $R$ is a rational map.  Based on   Lemma~\ref{lem:23orient} the last relation also implies that $\varphi$ is orientation-preserving,
because the maps $R$, $\psi$, and $f$ are.
\end{proof} 

The proof of Lemma~\ref{lem:pbackcostr} also leads to a criterion when a map on $S^2$ is a branched covering map.

\begin{cor} \label{cor:brcovcrit} Let $f\:S^2\ra S^2$ be a   continuous, open, and discrete map with  degree $\deg(f)>0$. Suppose that 
there exists a finite set $C\sub S^2$ such that $f$ is a local homeomorphism 
near each point in $S^2\setminus C$. Then $f$ is a branched covering map. 
\end{cor}

Recall that $\deg(f)\in \Z$ is the unique number such that 
$f_*([S^2])=\deg(f)  [S^2]$, where $[S^2]\in H_2(S^2)$ is the  fundamental class of $S^2$ (see Section~\ref{sec:orient}). 

A statement much stronger than
 Corollary~\ref{cor:brcovcrit} is actually true.
 
\begin{theorem} 
  \label{thm:Stoilow} 
  \index{branched covering map}
  \index{map!branched covering}
  Let $f\:S^2\ra S^2$ be a continuous, open, and light map    with  degree $\deg(f)>0$.  Then $f$ is a branched covering map. 
\end{theorem}

Here $f$ is called 
{\em light}\index{map!light}\index{light map} 
if $f^{-1}(p)$ is totally disconnected for each $p\in S^2$. 
Theorem~\ref{thm:Stoilow} follows from a  deep and more general characterization  theorem for  continuous, open,  and  light maps between surfaces (see \cite[Theorem~5.1, p.~198]{Why}; 
we are  grateful to P.~Ha\"{i}ssinsky for pointing out this
reference). Its  basic idea goes back to Sto\"\i low (see
\cite{St28} and \cite{LP17}). 

In particular, for  each 
continuous, open, and light map 
$f$  on $S^2$  we obtain the existence of local representations  as in \eqref{eq:brzd}, except that now the homeomorphisms 
$\psi$ and $\varphi$ are not necessarily orientation-preserving.
If $f$ has positive  degree in addition, 
then $f$ is actually  
a branched covering map according to our definition (as follows from the argument  below). For our purposes the weak form of Theorem~\ref{thm:Stoilow} as provided by Corollary~\ref{cor:brcovcrit} 
will suffice (it is used in the proof of 
Theorem~\ref{thm:f_descends_branched_cover}).

\begin{proof} [Proof of Corollary~\ref{cor:brcovcrit}] Let $p\in S^2$ be arbitrary and $q=f(p)$. The closed set  $f^{-1}(q)$ is discrete in $S^2$ and so necessarily finite. This and Lemma~\ref{lem:continv} imply that if we take a small enough topological disk 
$V$ with $q\in V$, then the unique component $U$ of $f^{-1}(V)$ that contains $p$ does not contain any other preimage of $q$ and no point 
in $C\setminus \{p\}$. 

By Lemma~\ref{lem:proper}~\ref{item:proper2} the map $f|U\: U\ra V$ is proper and it is clear by 
choice of $U$ and $V$ that the assumptions of Lemma~\ref{lem:z^d} are satisfied. So near $p$ the map $f$ has a representation as in \eqref{eq:z^dnonorrep}. In addition, the set of points where $f$ is not a local homeomorphism is finite. These assumptions are  enough  to argue  as in the proof of  Lemma~\ref{lem:pbackcostr} that  with suitable (and possibly different) conformal structures in source and target the map $f\: S^2\ra S^2$ is holomorphic. As  in the proof of Corollary~\ref{cor:brcovratup}, it follows   that there are homeomorphisms $\psi$ and $\varphi$ from  $S^2$ to $\CDach$ 
such that $R\coloneqq \psi\circ f\circ \varphi^{-1}$ is a rational map on $\CDach$. 

We may assume that $\psi$ is orientation-preserving, but without 
 additional assumptions we cannot guarantee that $\varphi$  has the same property.  Now if, as in our hypotheses,  we assume that  
 $\deg(f)>0$, then one can see that $\varphi$ is orientation-preserving 
 as follows. Let $[S^2]$ and $[\CDach]$ be the fundamental classes on
 $S^2$ and $\CDach$, respectively, and $\varphi_*$, $\psi_*$, and 
 $f_*$ the induced maps on homology of degree $2$. Then 
$R$ has positive   degree $\deg(R)>0$, because $R$ is rational, and $\psi_*([S^2])=[\Cdach]$, because $\psi$ is an orientation-preserving homeomorphism.
 Moreover, $\varphi_*([S^2])=\pm [\CDach]$ depending on whether 
 $\varphi$ preserves or reverses orientation. 
 Now  
 $$ \deg(f)[\Cdach]=\psi_* (f_*([S^2]))=R_*(\varphi_*([S^2]))=\pm \deg(R)
 [\Cdach]. $$
Since both $\deg(f)$ and $\deg(R)$ are positive, this implies that 
$\varphi_*([S^2])= [\CDach]$, and so $\varphi$ is indeed  orientation-preserving.

It is now clear that $f=\psi^{-1}\circ R\circ \varphi$ is a branched covering map. 
\end{proof}

\begin{lemma}[Compositions of branched covering maps]
  \label{lem:2_3_branched}
  \index{branched covering map}
  \index{map!branched covering}
  Let $X$, $Y$, and $Z$ be surfaces, and 
   $f\colon X\to Z$, $g\colon Y\to Z$, and 
  $h\colon X\to Y$ be continuous maps such that $f=g\circ h$. 
  \begin{enumerate}
  \item 
    \label{item:2_out3_1}
    If $g$ and $h$ are branched covering maps,  and $Y$ and $Z$ 
     are compact, then $f$ is also a branched covering map.
  \item 
    \label{item:2_out_3_2}
    If $f$ and $g$ are branched covering maps, then $h$ is a
    branched covering map. Similarly,  if $f$ and $h$ are branched covering maps, then $g$
    is a branched covering map.
  \end{enumerate}
  If, in addition, $X$, $Y$, $Z$ are Riemann surfaces and the two branched covering maps in the hypotheses of \ref{item:2_out3_1}
or \ref{item:2_out_3_2} are  holomorphic,  then the
  third map is also holomorphic.  
\end{lemma}
  

So in \ref{item:2_out_3_2} we can drop the assumption made in  \ref{item:2_out3_1} that $Y$ and  $Z$ are compact. On the other hand, in the last statement it is enough to assume that $g$ and $h$ are holomorphic and not necessarily branched covering maps in order to conclude that $f=g\circ h$ is holomorphic.

The surfaces and the maps in this statement  are related as in the following commutative diagram: 
\begin{equation*}
  \xymatrix{
    X \ar[r]^h \ar@(ur,ul)[rr]^f & Y \ar[r]^g & Z\rlap{.}
  }
\end{equation*}


\begin{proof}
\ref{item:2_out3_1}
We assume that $Y$ and $Z$ are compact, and both $g$ and $h$
are branched covering maps. 

This is the easiest case. First note that $g\: Y\ra Z$ is
finite-to-one, because $Y$ and $Z$ are compact. If $z_0\in Z$ is
arbitrary, we can find a small topological disk $W\subset Z$  that for the map $g$ is  an evenly covered neighborhood of 
 $z_0$  as in Definition~\ref{def:brcovmap}. Let $y_1, \dots,
 y_n\in Y$ be the preimage points of $z_0$ under $g$, and $V_1,
 \dots, V_n\subset Y$ be the components of 
$g^{-1}(W)$ with $y_i\in V_i$ for $i=1, \dots, n$. Each $V_i$ is a topological disk.  If we shrink $W$ to $z_0$, then each disk $V_i$ shrinks to $y_i$. So by Lemma~\ref{lem:evennei}  we may assume that $W$ is so small that for the map $h$ each  $V_i$ is an evenly covered neighborhood of $y_i$. Then, based on Lemma~\ref{lem:z^d}, one easily sees that $W$ is an evenly covered neighborhood of $z_0$ for the map
$f=g\circ h$. Hence $f$ is a branched covering map. 

If $X,Y$, and $Z$ are Riemann surfaces, and $g$ as well as $h$
are holomorphic, then $f$ is clearly holomorphic as well.  

\smallskip
(iia) We assume that $f$ and $g$ are branched covering maps. In contrast 
to case~\ref{item:2_out3_1} we do not assume that $Y$ and $Z$ are compact.  
 
 Let $y_0\in Y$ be arbitrary. In order to show that $h$ is a branched covering map, we have to find a neighborhood of $y_0$ that is evenly covered by $h$.
 
 For this we set $z_0=g(y_0)$. Since both $f$ and $g$ are
 branched covering maps, we can choose a small topological disk 
 $W\subset Z$ with  $z_0\in W$ that is evenly covered by 
 both $f$ and  $g$ (this follows from  Lemma~\ref{lem:evennei}). Then we have decompositions into connected components of the form 
 $$g^{-1}(W)=\bigcup_{j\in J}V_j\subset Y$$
and 
$$f^{-1}(W)=\bigcup_{i\in I}U_i \subset X$$
as in Definition~\ref{def:brcovmap}. Each set $V_j$ contains precisely one point 
$p_j\in g^{-1}(z_0)$, and each $U_i$ one point 
$q_i\in f^{-1}(z_0).$  

 We denote the connected component of $g^{-1}(W)$ that contains 
 $y_0\in g^{-1}(z_0)$ by  $V$; so $V=V_{j_0}$ and $y_0=p_{j_0}$  for some $j_0\in J$.  We claim 
that $V$ is either disjoint from the image $h(X)$  or is evenly covered by $h$.

To see this,  note that for each $i\in I$, the set 
$h(U_i)$ is connected and 
$$g(h(U_i))=f(U_i)=W;$$
so $h(U_i)\sub g^{-1}(W)$. Hence $h(U_i)$ has to lie in one of the connected components $V_j$ of $g^{-1}(W)$. Let $I_0$ be the possibly empty set of all $i\in I$ with $h(U_i)\sub V=V_{j_0}$. Then for each $i\in I\setminus I_0$ we have 
$h(U_i)\sub \bigcup_{j\in J\setminus\{j_0\}}V_j$, and so
$h(U_i)\cap V=\emptyset$.
This implies that 
$$ h^{-1}(V)=\bigcup_{i\in I_0}U_i.$$ 
So the sets $U_i$, $i\in I_0$, are the connected components of $h^{-1}(V)$. If $I_0=\emptyset$, then 
$h^{-1}(V)=\emptyset$, and so $V$ is disjoint from $h(X)$.

Now let us assume that $I_0\ne \emptyset$, and fix $i\in I_0$. We will verify the conditions in Lemma~\ref{lem:z^d} for the map $h\: U_i \ra V$. 
First note  that  
 $$h^{-1}(y_0)\cap U_i=
f^{-1}(z_0)\cap U_i= \{q_i\}. $$ The map $h|U_i\: U_i \ra V$ is proper. Indeed, if $K\sub V$ is compact, then 
$$ (h|U_i)^{-1}(K)=h^{-1}(K)\cap U_i\sub f^{-1}(g(K))\cap U_i= (f|U_i)^{-1}(g(K)). $$
Now $g(K)\sub W$ is compact and 
 $f|U_i\: U_i\ra W$ is proper. Hence 
 $(f|U_i)^{-1}(g(K))\sub U_i$ is compact.
 Since $\widetilde K\coloneqq (h|U_i)^{-1}(K)$ is relatively closed in $U_i$ and a subset of the compact set $(f|U_i)^{-1}(g(K))$,
 the set $\widetilde K$  is also compact. 
 
 Finally, if $q\in U_i\setminus \{q_i\}$, then $h$ is an orientation-preserving homeomorphism locally near $q$, because locally it is the composition of two such maps, namely, $f$ near $q$ followed by a local inverse of 
 $g$ near $p\coloneqq h(q)\ne y_0$ (note that this local  inverse of $g$ maps $f(q)=g(p)$ back to $p=h(q)$).  
 It now follows from Lemma~\ref{lem:z^d} that $V$ is evenly covered by $h$.
 
 We have seen that every point $y_0\in Y$ has a neighborhood $V$ that is evenly covered by $h$ or is disjoint from $h(X)$. This implies that $h(X)$ is a non-empty  open set with open complement. Since $Y$ is connected, it follows that  $h(X)=Y$. Therefore, actually every point in $Y$ has a neighborhood that is evenly covered by $h$. Hence $h$ is a branched covering map.     
 
 If, under the additional assumptions, $g$ and $f$ are holomorphic, then $h$ is holomorphic as well. Indeed, if $x_0\in X$ is not a critical point of $f$, then 
 $f$ is a local biholomorphism near $x_0$, and $g$ must be a local biholomorphism 
 near $y_0=h(x_0)$. Then $h$ is holomorphic near $x_0$, because near this point 
  $h$ is
 the composition of $f$ followed by a holomorphic  inverse branch $g^{-1}$ of $g$ with $g^{-1}(f(x_0))=y_0$.
 
  If $x_0$ is a  critical point of $f$, then there are no other critical points of $f$ nearby, and so $h$ is holomorphic in a punctured neighborhood of $x_0$. Since $h$  is continuous at $x_0$, this is a removable singularity for $h$
 and so $h$ is also holomorphic near  $x_0$. It follows that $h$ is holomorphic on $X$. 
 
 \smallskip 
(iib)
We assume that $f$ and $h$ are branched covering maps. Again we do not assume
that $Y$ and $Z$ are compact.  

 To see that $g$ is also a  branched covering map, let 
 $z_0\in Z$ be arbitrary. Then we can find a neighborhood $W$ of $z_0$ that is evenly covered by 
 $f$. Then again  we have a decomposition 
$$f^{-1}(W)=\bigcup_{i\in I}U_i\subset X$$
into connected components as in
 Definition~\ref{def:brcovmap}. Moreover, there exists a unique point $q_i\in U_i\cap f^{-1}(z_0)$ for each $i\in I$. 

Since $h$ is surjective, 
we have 
\begin{equation}\label{eq:ghsurj} 
 g^{-1}(W) = h(h^{-1} (g^{-1}(W)))=h(f^{-1}(W))=\bigcup_{i\in I} h(U_i). 
 \end{equation}  
Fix $i\in I$,   and consider the connected 
component $V\subset Y$ of the open set $g^{-1}(W)$ that contains $h(q_i)$. Since $h(q_i)\in h(U_i) \sub g^{-1}(W)$, it follows that $h(U_i)\sub V$. 
Note that the map $h|U_i\: U_i \ra V$ is proper. Indeed, if $K\sub V$ is compact, then 
$$(h|U_i)^{-1}(K) \sub (f|U_i)^{-1}(g(K)), $$
 and, by a similar reasoning as in the proof of (iia), we see that the set 
 $(h|U_i)^{-1}(K)$ is also compact. So we can apply 
 Lemma~\ref{lem:proper}~\ref{item:proper1} to the map $h|U_i\: U_i\ra V$ and it  follows that $V=h(U_i)$. 
 So for each $i\in I$ the set $h(U_i)$ is the connected component of $g^{-1}(W)$ that contains $h(q_i)$.
 Moreover, by \eqref{eq:ghsurj} each connected component $V$ of  $g^{-1}(W)$ is such an image $h(U_i)$ for some $i\in I$. 
 
We now claim that $W$ is evenly covered by $g$. 
To see this, let $V$ be a connected component of
 $g^{-1}(W)$. It suffices to verify the conditions of Lemma~\ref{lem:z^d} 
for the map $g|V\: V \ra W$. For this we pick  $U_i$ such that $h(U_i)=V$, and let $p_0=h(q_i)$. 

Then 
$$ g^{-1}(z_0)\cap V= h(f^{-1}(z_0)\cap U_i)=\{ h(q_i)\}=\{ p_0\}. $$
The map $g|V \: V \ra W$ is proper, because if $K\sub W$ is compact, 
then $f^{-1}(K)\cap U_i=(f|U_i)^{-1}(K)$ is compact, and so
$$ (g|V)^{-1}(K)=g^{-1}(K)\cap V=h(f^{-1}(K)\cap U_i)$$ is compact as well. 
 
  Finally, if $p\in V\setminus\{p_0\}$, then $g$ is an orientation-preserving homeomorphism near 
 $p$; again this is true, because near $p$ the map $g$ is the composition of two such maps. To see this, we pick 
 $q\in U_i$ with  $h(q)=p$. Then $q\ne q_i$, and so 
 $f$ is an orientation-preserving homeomorphism near 
 $q$. Since  $f=g\circ h$, the map $h$ is also locally injective near $q$, and hence  an orientation-preserving homeomorphism near $q$, because $h$ is a branched covering map. A local inverse $h^{-1}$ of  $h$ defined near $p$ with $h^{-1}(p)=q$ has the same property.
 Finally, 
 $g=f\circ h^{-1}$ near $p$ and we have a local representation of $g$ as desired. 
 It follows that $g$ is indeed a branched covering map. 
 
 If, under the additional assumptions, $f$ and $h$ are
 holomorphic, then $g$ is also holomorphic as can be verified by
 an argument very similar to the one in (iia): one first shows that $g$ is holomorphic away from the critical values of $h$ and then  that the critical values of $h$ are removable singularities for $g$. 
\end{proof}

As the following lemma shows, the local degree behaves in the expected way under compositions of branched covering maps.

\begin{lemma}
  \label{lem:locdeg}
  \index{local degree}\index{deg@$\deg_f(p), \deg(f,z)$} 
  Let $f\: X\ra Y$ and $g\: Y\ra Z$ be branched covering maps, where 
 $X$, $Y$, and $Z$ are  surfaces such that  $Y$ and $Z$ are compact.
  Then 
 \begin{equation}\label {eq:locdegcomp} 
  \deg(g\circ f, x)= \deg(g, f(x))\cdot \deg(f,x)
  \end{equation}
  for all $x\in X$. 
  \end{lemma}
  
  \begin{proof} Let $x\in X$ be arbitrary.
    We know by Lemma~\ref{lem:2_3_branched} that $g\circ f$ is a branched covering map, and so all terms in \eqref{eq:locdegcomp} are defined. 
    
 Let  $y=f(x)$ and 
  $z=g(y)=g(f(x))$. We consider  a point $z'\ne z$ 
   close to $z$. In order to determine $\deg(g\circ f, x)$, we have to count  the number of preimages of $z'$ under $g\circ f$ that are close to $x$. 
  
 Now if $z'\ne z$ is close to $z$, then it has 
  $k\coloneqq \deg(g, y)$ distinct preimages $y_1, \dots , y_k\ne y$ under $g$ near 
  $y$.  We may assume that $z'$ is so close to $z$ that 
  the points $y_1, \dots , y_k$ lie in a neighborhood of $y$ that is evenly covered by $f$. Then each of the points $y_i$, $i=1, \dots, k$,  has precisely 
  $l\coloneqq \deg(f,x)$  preimages under $f$ close to $x$. 
  
 So $z'$ has precisely $k\cdot l=\deg(g, f(x))\cdot \deg(f,x)$ preimages under $g\circ f$ that are close to $x$ and \eqref{eq:locdegcomp}
 follows. 
 \end{proof}

 We next discuss existence and uniqueness for lifts by branched covering maps.
 The basic diagram is again \eqref{eq:basicliftdia}, where $\pi$ is now a given branched 
 covering map. A map $g$ as in \eqref{eq:basicliftdia}, is called a {\em lift} of $f$ (by $\pi$). 
 We first record an important special case of this situation.

 \begin{lemma}[Lifting paths by branched covering maps] 
 \label{lem:liftsofpathsbranched}\index{branched covering map!lift of path by}\index{lift!of path by branched covering map}
Let $X$ and $Y$ be surfaces,  $\pi\: X\ra Y$ be a branched covering map,  $\ga\: [0,1]\ra Y$ be a path in $Y$, and $x_0\in \pi^{-1}(\ga(0))$. Then there exists a path $\alpha\: [0,1]\ra X$ with 
$\alpha(0)=x_0$ and $\pi\circ \alpha=\ga$.  
\end{lemma}  

So every path can be lifted by a branched covering map. Moreover, any point in the fiber over the initial point of the original path can be prescribed as the initial point of the lift. In general, the lift is not uniquely determined, because it can make various ``turns''
as it runs through  critical points of $\pi$.  For a  more general path-lifting statement see \cite[Theorem~2]{Fl50}. 

\begin{proof} We will only give an outline of the proof 
 and leave some  straightforward details to the reader. 

If we break up $\ga$ into small subpaths, we can  reduce to the situation where $\ga$ runs in an open subset of $Y$ that is evenly covered by $\pi$. Using local coordinates, we can further make the assumption  that  $X$ and $Y$ are equal  to the open unit disk $\D$ in $\C$, and $\pi(z)=z^n$ for $z\in X=\D$, where $n\in \N$.  

Then the  set $[0,1]\setminus \ga^{-1}(0)$ can be written as a disjoint union  
of  relatively open intervals  $I_j\sub [0,1]$, $j\in J $, where $J$ is a countable index set. 
Each path  $\ga|I_j$ runs in 
$\D\setminus\{0\}$ and hence has a (non-unique)  lift $\alpha_j$ by  the covering map $\pi\: \D\setminus\{0\}
\ra \D\setminus\{0\}$. Moreover, if $0\in I_j$ we can choose the lift $\alpha_j $ so that $\alpha_j(0)=x_0$ for a given point $x_0\in \pi^{-1}(\ga(0))$.  

 A continuous lift $\alpha\:[0,1]\ra \D$ of $\ga$ with $\alpha (0)=x_0$ is now defined  by setting $\alpha(t)=\alpha_j(t)$ if $t\in [0,1]$ lies in one of the intervals $I_j$, and setting $\alpha(t)=0$ if not. 
\end{proof} 

The following statement is the  analog of Lemma~\ref{lem:liftsofcov} for branched covering maps.

\begin{lemma}[Lifting  branched covering maps] 
\label{lem:liftsofbranched}
\index{lift!of map!by branched covering map}
\index{branched covering map!lift of map by}
Let $X$, $Y$, and $Z$ be surfaces,  and $\pi\: X\ra Y$ be a branched covering map. 

 \begin{enumerate}
 
   \item
    \label{item:Liftuniq}
   Suppose $g_1, g_2\: Z\ra X$ are two continuous and discrete maps such that $\pi\circ g_1=\pi\circ g_2$. If there exists a point $z_0\in Z$ such that 
 $g_1(z_0)=g_2(z_0)=:\!x_0$,  and $\pi(x_0)\in Y\setminus\pi(\crit(\pi))$, then $g_1=g_2$.  
   
    \smallskip
  \item 
    \label{item:Liftex}
    Suppose  $Z$ is simply connected,  $f\: Z\ra Y$ is a branched covering  map, and  $x_0\in X$ and   $z_0\in Z$  are points such that $\pi(x_0)=f(z_0)$. 
    
    If for all  $x\in X$ and  $z\in Z$ with $\pi(x)=f(z)$  we have $$\deg(\pi,x) |\deg(f,z), $$ then there exists a branched covering map
    $g\: Z\ra X$ such that $g(z_0)=x_0$ and  $f=\pi\circ g$.

     \end{enumerate}
\end{lemma}  
The situation is illustrated in the following commutative diagram:
\begin{equation*}
  \xymatrix{
    & X \ar[d]^\pi
    \\
    Z\ar[ur]^g \ar[r]^f & Y\rlap{.}
    }
\end{equation*}

If $\pi$ is a branched covering map as in the  statement, then we
call a fiber $\pi^{-1}(y_0)$, $y_0\in Y$, 
{\em clean}\index{clean fiber} 
if it does not contain critical points of $\pi$, or equivalently, if $y_0$ is not a critical value of $\pi$. 

The maps $g_1$ and $g_2$ in \ref{item:Liftuniq} are lifts of $f\coloneqq \pi\circ g_1=\pi\circ g_2$ by  $\pi$. Moreover, $y_0=\pi(x_0)\in Y\setminus 
\pi(\crit(\pi))$ is not a critical value of $\pi$ and  so $\pi^{-1}(y_0)\cap \crit(\pi)=\emptyset$. This means that  $x_0\in 
\pi^{-1}(y_0)$ lies in a clean fiber of $\pi$. So
\ref{item:Liftuniq} says that under our hypotheses  lifts are uniquely determined by the image of a point that maps 
into a clean fiber of $\pi$. 

\begin{proof}  \ref{item:Liftuniq} The maps $g_1$ and $g_2$ are lifts of $f =\pi\circ g_1=\pi\circ g_2$ by  the branched covering map $\pi\: X\ra Y$. We are claiming that with the given normalization  such  a lift is unique. 

To see this, let  $P_Y\coloneqq \pi(\crit(\pi))$, $P_X \coloneqq \pi^{-1}(P_Y)$,  
$P_Z \coloneqq f^{-1}(P_Y)=g_1^{-1}(P_X)=
g_2^{-1}(P_X)$. 
Note that these are discrete sets in  $Y$, $X$,  and $Z$, respectively. 
Let $X'=X\setminus P_X$, $Y'=Y\setminus P_Y$, and $Z'=Z\setminus P_Z$.
Then $x_0\in X'$, $z_0\in Z'$, $f(Z')\sub Y'=\pi (X')$, and
$g_1(Z')=g_2(Z')\sub X'$. Moreover, if we restrict $\pi$ to $X'$,  then we
obtain a covering map in the usual sense. Now by 
Lemma~\ref{lem:liftsofcov}~\ref{item:liuniq} a lift of  a continuous
map  on a connected surface by   a covering map is  uniquely determined
by the image of one point. Since $P_Z$ is a discrete set, $Z'$ is a connected surface. So $g_1(z_0)=x_0=g_2(z_0)$ implies that 
$g_1|Z'=g_2|Z'$. Moreover,   $Z'$ is dense in $Z$, and so  it follows that $g_1=g_2$ as desired. 

\smallskip 
\ref{item:Liftex} This is an existence statement for lifts by 
branched covering maps. It  will be  derived  from an analog of  the
classical monodromy  theorem in complex analysis 
(see \cite[Section~1.7]{Fo81}). We will only give an outline of the argument and leave some details to the reader. 

We choose a conformal structure on $Y$.
Then by Lemma~\ref{lem:pbackcostr} there exist conformal structures on  $X$ and $Z$ such that the maps $\pi$ and $f$ are holomorphic if we equip the surfaces with these conformal structures.

A {\em germ} (of a lift of $f$)  at a point $z\in Z$ is a holomorphic  map 
$g_z\: U_z \ra X$ defined on an open and connected neighborhood 
$U_z\sub Z$ of $z$ such that $f|U_z= \pi \circ g_z$. We consider two such germs at $z$ as equivalent  if they agree on a neighborhood of $z$. Actually, one should define a germ   as  an equivalence class of such local lifts at $z$, but for the sake of easier exposition it is  convenient to  ignore the distinction between a local lift and the equivalence class  that it represents.  

 The value $g_z(z)$ does not determine $g_z$ uniquely in general, because the fiber $\pi^{-1}(f(z))$ may not be clean; this is true for nearby points $z'$ and so  it follows from 
\ref{item:Liftuniq} that the value $g_z(z')$ for a point $z'\ne z$ sufficiently close to $z$ determines 
$g_z$ uniquely.

This fact allows one to define a notion of an {\em  analytic continuation}  of a germ 
$g_z$ 
along a path $\alpha\: [0,1]\ra  Z$ with $\alpha(0)=z$. Such a continuation is given by   germs 
$g_{\alpha(t)}$ at $\alpha(t)$ for $t\in [0,1]$ such that
$g_{\alpha(0)}$ 
is equivalent to
the given germ $g_z$. Moreover,  we require  that these germs  are  compatible in the following sense: 
if  $t_0\in [0,1]$ is arbitrary, then each  germ $g_{\alpha(t)}$ for  $t\in [0,1]$ close enough to $t_0$ 
is equivalent to a germ obtained by restricting $ g_{\alpha(t_0)}$ to a suitable neighborhood of $\alpha(t)$. 

As in the classical monodromy theorem, one can  show that if an analytic continuation of $g_z$ along a path  $\alpha\: [0,1]\ra  X$ exists, then up to equivalence 
the germ $g_{\alpha(0)}$ uniquely determines $g_{\alpha(1)}$. In this case we say that 
the (equivalence class of the) germ $g_{\alpha(1)}$ is obtained by 
analytic continuation of $g_{\alpha(0)}$ along $\alpha$.

Analytic continuation of a germ along homotopic paths with the same endpoints leads  to  equivalent  germs.
More precisely, let   $z,z'\in Z$,  $H\: [0,1]\times [0,1]\ra Z$ be  continuous, and assume that $z=H(s,0)$  and $z'=H(s,1)$
for all $s\in [0,1]$. Suppose that we have an analytic continuation of a germ $g_z$ at $z$ along every path $t\mapsto \alpha^s(t)\coloneqq H(s,t)$, $s\in [0,1]$. Let  $g^0_{z'}$ and $g^1_{z'}$ be the two germs at $z'$  obtained by analytic continuation of $g_z$ along $\alpha^0$ and $\alpha^1$, respectively. Then $g^0_{z'}$ and $g^1_{z'}$ are equivalent. 

Now suppose $x\in X$ and $z\in Z$ are points with 
$y\coloneqq \pi(x)=f(z)$.
Our hypotheses 
imply that  then  $d\coloneqq \deg(\pi,x)$ divides $k\coloneqq \deg(f,z)$. Let $m\coloneqq k/d\in \N$. If we choose suitable  local conformal  coordinates near   $x$, $y$, and $z$, then we can represent the map $\pi $ near $x$  and the map $f$ near $z$ by the power maps $P_d$ and $P_k$ near $0$, respectively. Here we use the notation $P_l(u)=u^l$ for $l\in \N$, $u\in \C$. 
Note that $P_k=P_d\circ P_m$.

We may assume that the conformal
coordinates near $x$ and $z$ are defined on topological disks $U$ and $V$, respectively. If $g$ is any holomorphic germ of a lift of $f$ defined on a subregion $V'\sub V$ and mapping into $U$, then in the 
given conformal coordinates we have $P_k=P_d\circ P_m= P_d\circ g$.
This implies that $g=cP_m$ on $V'$, where $c\in \C$ and $c^d=1$. Conversely, each map $g=cP_m$ of this form clearly satisfies $P_k=P_d\circ g$.
 In  particular, the given germ $g$ on $V'$ can actually be extended to a germ defined on the whole neighborhood 
$V$ of $z$. 

These considerations show that first of all we can find a germ $g_{z_0}$ defined near $z_0$ such that $g_{z_0}
(z_0)=x_0$. Moreover, the  analytic continuation of $g_{z_0}$ 
 along any path $\alpha\: [0,1]\ra X$ with $\alpha(0)=z_0$ exists, because we never encounter any singularities   along $\alpha$
 beyond which we cannot extend our local germs.

 The fact that $Z$ is simply connected and the homotopy invariance of analytic continuation as discussed 
 guarantees that up to equivalence the initial choice of the germ $g_{z_0}$ leads to a unique germ $g_z$ at every point $z\in Z$. These germs are locally compatible. So  if we define $g(z)=g_z(z)$ for $z\in Z$, then $g$  is a holomorphic map 
 from $Z$ to $X$ such that $g(z_0)=x_0$ and $f=\pi\circ g$. 
 Since $\pi$ and $f$ are branched covering maps, 
 Lemma~\ref{lem:2_3_branched}~\ref{item:2_out_3_2} implies that $g$ is a branched covering map as well. 
\end{proof}

\section{Quotient spaces and group actions}
\label{sec:appquotmaps} 
\index{quotient}

In this section we discuss some 
general facts about quotient spaces, group actions, and how we can pass maps to quotients.
All of this material is fairly standard. 

 Let $\sim$ be an equivalence relation on a set $X$. We denote by
$X/\Sim$ the 
{\em quotient space}\index{quotient!space}\index{Xsim@$X/\Sim$}
consisting of all equivalence
classes $[x]\coloneqq \{y\in X: x\sim y\}$ for $x\in X$, and by $\pi \colon X\to
X/\Sim$ the quotient map that sends each point $x\in X$ to its
equivalence class $[x]$. If $X$ is a topological space, then we equip
$X/\Sim$ with the 
{\em quotient topology}.\index{quotient!topology}
In this topology a set
$U\sub X/\Sim$ is open if and only if $\pi^{-1}(U)\subset X$ is
open. Then 
$\pi$ is a continuous map. Moreover, a map $f\:X/\Sim \ra Z$ into
another topological space $Z$ is continuous if and only if $f\circ \pi
\: X \ra Z$ is continuous. Actually, this functorial property
characterizes the quotient topology on $X/\Sim$.

Let  $\Theta\colon X \to Y$ be  a map between sets $X$
and $Y$. The equivalence relation on $X$ {\em induced} by
  $\Theta$\index{equivalence relation!induced by!map} is defined by
\begin{equation*}
  x\sim y \;:\Leftrightarrow \;\Theta(x) = \Theta(y) 
\end{equation*}
for $x,y\in X$. Clearly, this is indeed an equivalence
relation on $X$. The following well-known fact allows one to
identify the quotient $X/\Sim$ with $Y$ in the cases that are
relevant for us.

\begin{lemma}
  \label{lem:XsimY}
  Let $\Theta\colon X\to Y$ be a continuous and surjective map
  between topological spaces $X$ and $Y$. Let $\sim$ be the
  equivalence relation on $X$ induced by $\Theta$ and $\pi\colon
  X\to X/\Sim$ be the quotient map. Assume that
  \begin{enumerate}
  \item
    \label{item:XsimY_map_open}
    $\Theta$ is an open map, or
  \item 
    \label{item:XsimY_cpt_T2}
    $X/\Sim$ is compact and $Y$ is Hausdorff.
  \end{enumerate}
  Then there exists a unique homeomorphism $\varphi\colon
  X/\Sim \to Y$ such that $\Theta = \varphi \circ \pi$.
\end{lemma}
 
\begin{proof} 
  We define $\varphi([x])=\Theta(x)$ for $[x]\in X/\Sim$. Since
  $\sim $ is induced by $\Theta$, it is clear that
  $\varphi\: X/\Sim\ra Y$ is a well-defined map satisfying
  $\Theta=\varphi\circ \pi$. This last identity determines
  $\varphi$ uniquely. The map $\varphi$ is injective, as follows
  from the fact that $\sim$ is induced by $\Theta$. It is also
  surjective, because $\Theta$ is, and continuous by the
  functorial property of the quotient topology.  So $\varphi$ is
  a continuous bijection.
  
  Let us first assume that $\Theta$ is an open map. To show
  that $\varphi$ is a homeomorphism, it suffices to show that
  $\varphi$ is an open map as well. To see this, let  
  $U\sub X/\Sim$ be an arbitrary open set. Then $V\coloneqq
  \pi^{-1}(U)\sub X$ is also open. Since 
  $$  \varphi(U)=\varphi(\pi(V))=\Theta(V)$$ and 
  $\Theta$ is an open map, it follows that $\varphi(U)\subset Y$
  is open as desired.

 Suppose  as in (ii) that $X/\Sim$ is compact and $Y$
  is Hausdorff. Let $A\subset X/\Sim$ be an arbitrary closed,
  and hence compact, set. Then $\varphi(A) \subset Y$ is
  compact; so this set is closed since $Y$ is Hausdorff. Thus
  $\varphi$ is a closed bijection which implies that
  $\varphi^{-1}$ is continuous. The first part of the proof
  implies that $\varphi$ is a homeomorphism as desired.
\end{proof}

We now study when maps descend to quotients. Let $\sim$ be an
equivalence relation on a set $X$ and $\pi\: X\ra X/\Sim$ be
the quotient map. If $f\colon X\to X$ is a map, then we say
that $f$ 
{\em descends}\index{map!descending to quotient}\index{quotient!map descending to}\index{descending!to quotient}
to the quotient $X/\Sim$ if there
exists a map $\widetilde f\: X/\Sim \ra X/\Sim$ such that
$\widetilde f \circ \pi = \pi \circ f$. In this case, we have
the following commutative diagram:
\begin{equation*}
  \xymatrix{
    X \ar[r]^f \ar[d]_\pi & X \ar[d]^\pi
    \\
    X/\Sim \ar[r]^{\widetilde f} & X/\Sim\rlap{.}
  }
\end{equation*}
Here necessarily $\widetilde f([x])=[f(x)]$ for $x\in X$. This
shows that if $f$ descends to $X/\Sim$, then $\widetilde f$ is
uniquely determined. 

It is easy to characterize when a map $f$ descends. The relevant
condition is that $\sim$ should be {\em invariant} under
$f$ (or \emph{$f$-invariant}),\index{equivalence relation!f-invariant@$f$-invariant}\index{f-invariant@$f$-invariant!equivalence relation}\index{invariant!equivalence relation}
which means that the
implication  
\begin{equation*}
  x\sim y \Rightarrow f(x) \sim f(y)
\end{equation*}
is satisfied  for all $x,y\in X$. 

\begin{lemma}[Quotients of maps]
  \label{lem:f_descends}
  Let $\sim$ be an equivalence relation on a set $X$ with the quotient map $\pi\: X\ra 
  X/\Sim$. Suppose $f\colon X\to X$ is a map.
  \begin{enumerate}
  
  \smallskip
  \item 
\label{item:f_desc_sim_inv}
Then $f$ descends to a map $\widetilde f$ on $X/\Sim$ if and
only if $\sim$ is in\-variant under $f$.
    
  \smallskip
  \item 
    \label{item:f_desc_cont}
   Suppose   $X$ is a topological space, and  $f$ is a continuous map that descends to the map $\widetilde f\: X/\Sim\,\ra X/\Sim$. Then $\widetilde f$ is continuous.     
    
  \end{enumerate}
\end{lemma}

\begin{proof}
  \ref{item:f_desc_sim_inv} 
  Suppose $\sim$ is invariant under $f$. Then 
  $\widetilde f\colon X/\Sim \ra X/\Sim$ given by
  $\widetilde f([x])\coloneqq [f(x)]$ for $[x]\in X/\Sim$ is
  well-defined. Obviously,
  $\widetilde f \circ \pi = \pi \circ f$ which shows that $f$
  descends to $X/\Sim$.
  
  Conversely, suppose that $f$ descends to the map
  $\widetilde f$ on $X/\Sim$. Let $x,y\in X$ with $x\sim y$ be
  arbitrary. Then $[x]=[y]$, and so
  $$ [f(x)]=(\pi\circ f)(x)=(\widetilde f\circ \pi)(x)=\widetilde f([x])=\widetilde f([y])=[f(y)].$$ 
  Hence $f(x)\sim f(y)$, and we see that $\sim$ is invariant
  under $f$.
  
   \smallskip
  \ref{item:f_desc_cont}
  If
  the continuous map $f\: X \ra X$ descends to the map
  $\widetilde f\: X/\Sim \ra X/\Sim$, then
  $\widetilde f\circ \pi =\pi \circ f$ is continuous. 
  Hence
  $\widetilde f$ is continuous by the functorial property of the
  quotient topology.
\end{proof}

A more general continuity criterion related to the second part of the previous lemma can be formulated as follows. 

\begin{lemma} \label{lem:f_desc_cont} 
Let $X$ and $Y$ be topological spaces, and $A\: X\ra X$, $\Theta\: X\ra Y$, 
$f\: Y\ra Y$ be maps with $f\circ \Theta =\Theta \circ A$. Suppose $A$ is continuous, and $\Theta$ is continuous, surjective, and open. Then $f$ is continuous.
\end{lemma} 

\begin{proof} 
  Suppose $W\subset Y$ is open. We define
  $V\coloneqq f^{-1}(W)\subset Y$ and
  $U\coloneqq \Theta^{-1}(V)\sub X$. Then
  $\Theta\circ A= f\circ \Theta$ implies that
  \begin{equation*}
    U
    = 
    \Theta^{-1}(V)  
    =  
    \Theta^{-1}(f^{-1}(W))
    =
    A^{-1}(\Theta^{-1}(W)).
  \end{equation*}
  Since $A$ and $\Theta$ are continuous, it follows that $U$ is
  open. Now $\Theta(U) = V$, since $\Theta$ is
  surjective. Moreover, $\Theta$ is open which implies that
  $V=\Theta(U)=f^{-1}(W) $ is open. Hence $f$ is continuous as
  desired.
\end{proof}

We now consider equivalence relations that are induced by group
actions. We first  review some general   terminology.

Let $X$ be a topological space and $G$ be a group of homeomorphisms
acting on $G$. The 
equivalence relation $\sim$ on $X$ \emph{induced} by
  $G$\index{equivalence relation!induced by!group}
 is defined by
\begin{equation}
  \label{eq:eq_ind_G}
  x\sim y \;:\Leftrightarrow \;\text{there exists $g\in G$ such that $y=g(x)$}
\end{equation}
for $x,y\in X$. The equivalence class $[x]$ of a point $x\in X$ with
respect to $G$ is equal to its 
{\em $G$-orbit}\index{orbit!by group action}
$Gx\coloneqq \{g(x):g\in G\}$ under
$G$.  We denote the quotient space by 
$X/G$,\index{Xx@$X/G$}
and equip it with the
quotient topology. As before, we denote by $\pi$ the quotient map
$\pi\: X\ra X/G$ given by $\pi(x)=[x]=Gx$ for $x\in X$.

A 
{\em fundamental domain}\index{fundamental domain} 
for the action of $G$ is a closed set $F\sub X$
such that 
every orbit of $G$ has at least one point in $F$, and at most one
point in the interior $\inte(F)$. If $F$ is a fundamental domain for $G$, then, at
least on an intuitive level, one obtains the quotient space $X/G$ from
$F$ by identifying the points on the boundary of $F$ that lie in the
same orbit.

The 
{\em stabilizer}\index{stabilizer}
$G_x$ of a point $x\in X$ is the subgroup of $G$
consisting of all elements $g\in G$ with $g(x)=x$.  We call the action
{\em cocompact}\index{cocompact}\index{group action!cocompact}
if there exists a compact set $K\sub X$ such that
$X=\bigcup_{g\in G} g(K)$. So then the images  of $K$ under
the elements in  $G$
cover $X$, and $X/G=\pi(K)$. In particular, the quotient space $X/G$
is compact as it is the continuous image $\pi(K)$ of the compact set $K$.
 

 The action of $G$ on $X$ is
called 
{\em properly discontinuous}\index{group action!properly discontinuous}\index{properly discontinuous}  
if for each compact set $K\sub X$
the set $\{g\in G: g(K)\cap K\ne \emptyset\}$ is finite. 
Often one makes additional assumptions on $X$ here that ensure a supply of sufficiently many compact subsets $K\sub X$. This is the case, for example, if  $X$ is a metric space that is
 {\em proper},\index{proper!metric space}\index{metric!proper} 
 i.e., closed balls in $X$ are compact. 
We call the
action of $G$ on a metric space $X$ 
{\em geometric}\index{group action!geometric}\index{geometric group action}  
if it is cocompact and properly
discontinuous, and each element of $G$ acts as an isometry on
$X$.

Suppose $Y$ is another topological space and $\Theta\: X\ra Y$ is a
continuous  map. We say that $\Theta$ is 
{\em induced}\index{map!induced by group action}\index{induced by group action}
by the action of a group $G$ acting on $X$ if $\Theta(x)=\Theta(y)$
for $x,y\in X$ if and only if there exists $g\in G$ such
$y=g(x)$. In this case, the equivalence relation on $X$ induced
by $\Theta$ is the same as the equivalence relation induced by
$G$. 

The following statement  immediately follows  
from Lem\-ma~\ref{lem:XsimY}.
 
\begin{cor}
  \label{cor:groupquot} 
  Let $X$ and $Y$ be topological spaces, $G$ be a group acting
  on $X$, $\pi\: X\ra X/G$ be the quotient map, and
  $\Theta\: X \ra Y$ be a  continuous and surjective  map induced
  by $G$. Assume
  \begin{enumerate}
  \item
    \label{item:groupquot1}
    $\Theta$ is an open map, or 
  \item
    \label{item:groupquot2} 
    $G$ acts cocompactly on
  $X$ and $Y$ is Hausdorff.
  \end{enumerate}
  Then there exists a unique homeomorphism $\varphi \: X/G\ra Y$
  such that $\Theta=\varphi\circ \pi$.
\end{cor}

Note that in case \ref{item:groupquot2} the quotient space $X/G$
is compact and so Lemma~\ref{lem:XsimY} indeed applies.
  
Under the assumptions of  Corollary~\ref{cor:groupquot}   one can identify $Y$ with the quotient space $X/G$ by the homeomorphism $\varphi$,  and
$\Theta$ with the quotient map $\pi\: X \ra X/G$.
  



Let us now consider when a map $f\colon X\to X$ descends to the
quotient $X/G$.  
A relevant condition  
is that $f$ is 
{\em $G$-equivariant}.\index{equivariant|textbf}\index{group action!map equivariant under|textbf}\index{G-equivariant@$G$-equivariant|textbf}
This means  that for each $g\in G$ there exists $h\in G$ such that 
 \begin{equation}\label{eq:fequivariant}
 f\circ g=h\circ f. 
 \end{equation}
If $f\: X\ra X$ is a bijection, then obviously $f$ is   $G$-equivariant  if and only if 
 \begin{equation}\label{eq:fequivariant2}
 f\circ g \circ f^{-1}\in G \text{ for each $g\in G$}. 
 \end{equation}

 We now formulate a  criterion in  a special  case relevant  for us.

\begin{lemma}
  \label{lem:f_groupdescend} Let $X$ and $Y$ be surfaces, $f\: X\ra X$ be a homeomorphism, and $\Theta\: X\ra Y$ be a branched covering map induced by an action of a group $G$ on $X$. Then there exists a continuous map 
  $\widetilde f\: Y\ra Y $ such that $\Theta\circ f=\widetilde f\circ \Theta$ if and only if $f$ is $G$-equivariant.   
\end{lemma}

Note that $\Theta$ is an open map, and so by Corollary~\ref{cor:groupquot} we can identify $Y$ with the quotient space $X/G$.
Under this identification the lemma gives a condition for the homeomorphism  $f$ to descend to $X/G$. 

As we will see in the proof, the ``if'' implication in this statement is valid  in much greater generality. 

\begin{proof} Suppose first that $f$ is $G$-equivariant. 
Let $\sim$ be the equivalence relation on $X$ induced by $G$. Then by 
 Corollary~\ref{cor:groupquot} and Lemma~\ref{lem:f_descends} it suffices to show that 
 $\sim$ is invariant under $f$. So let  $x,y\in X$ and $x\sim y$. Then there exists $g\in G$ such that $y=g(x)$. Since $f$ is  $G$-equivariant, there exists $h\in G$ such that $f\circ g=h\circ f$. Hence 
$$f(y)=f(g(x))=h(f(x))\sim f(x)$$ 
as desired.

To prove the other implication, 
we assume that a continuous map $\widetilde f\: Y\ra Y$ with $\Theta\circ f=\widetilde f \circ \Theta$ exists. In order to show that $f$ is $G$-equivariant, let $g\in G$ be arbitrary. We can pick  a point $x_0\in X$ so that the homeomorphism $f\circ g$ maps $x_0$  
to  a clean fiber of $\Theta$, i.e.,  $(\Theta \circ f\circ g)(x_0)$ is not a critical value of $\Theta$. Then 
\begin{align*} 
(\Theta\circ f\circ g)(x_0)&= (\widetilde f \circ \Theta\circ g)(x_0)\\
&= (\widetilde f \circ \Theta)(x_0)= (\Theta\circ f)(x_0). \nonumber 
\end{align*} 
Since $\Theta$ is induced by $G$, it follows that there exists $h\in G$ such that 
\begin{equation} \label{eq:clx0} f (g(x_0))=h(f(x_0)). \end{equation}

Now 
$$ 
\Theta\circ f\circ g = \widetilde f \circ \Theta\circ g 
= \widetilde f \circ \Theta = \Theta\circ f= \Theta \circ h \circ f, 
$$
and so by \eqref{eq:clx0} the homeomorphisms $f\circ g$ and $h\circ f$ satisfy the conditions of Lemma~\ref{lem:liftsofbranched}~\ref{item:Liftuniq} for the branched covering map $\Theta$. Hence 
$f\circ g=h\circ f$. This shows that $f$ is $G$-equivariant.\end{proof}

    \section{Lattices and tori}
\label{sec:applifttorend}

In this section we review some facts about lattices and tori.

A {\em lattice}\index{lattice}  
$\Gamma \sub \R^2$ is a non-trivial discrete subgroup of $\R^2$ (considered as a group with vector addition). 
Often it is more convenient to consider a lattice as a discrete (additive)  subgroup of $\C\cong \R^2$. In the following,  we will freely switch back and forth between these different viewpoints and real and complex notation.  

The {\em rank}\index{rank of lattice}
 of a lattice is the dimension of the subspace of $\R^2$ 
(considered as a  real vector space) spanned by the elements in
$\Gamma$. The lattice can be a {\em rank-1 lattice}. Then, using   complex notation, we can write it in the form $\Gamma= \Z \omega$, where  $\omega\in \C\setminus \{0\}$.  
 The other case is that 
$\Gamma$ is a 
{\em rank-2 lattice}. Then there exist  generators 
$\omega_1,\omega_2\in \C\setminus \{0\}$ with $\imag(\omega_2/\omega_1)>0$ such that 
$$ \Gamma= \Z  \omega_1\oplus \Z  \omega_2\coloneqq  \{ m\omega_1+n \omega_2: m,n\in \Z\}. $$ 

 If $\Gamma \sub \R^2\cong \C$ is  a rank-$2$ lattice, then we have  an 
equivalence relation $\sim$ on $\R^2$  given by 
\begin{equation} \label{eq:latteq} 
x\sim y :\Leftrightarrow x-y\in \Gamma
\end{equation} 
for $x,y\in \R^2$. We denote the quotient space $\R^2/\Sim$ by $\R^2/\Gamma$,
 and by $\pi\: \R^2\ra \R^2/\Gamma$ the quotient map that sends a point $x\in \R^2$ to its equivalence class $[x]$. In complex notation,  we use $\C/\Gamma$ to denote the quotient space; the quotient map is then a map $\pi \: \C \ra \C /\Gamma$.
 
For each $\ga\in \Gamma$ we can define  an associated  translation 
 $$\tau_{\ga}\: \R^2\ra \R^2, \ u\in \R^2 \mapsto \tau_\ga(u)\coloneqq  u+\ga.$$ The translations  $\tau_\ga$, $\ga \in \Gamma$, form a group under composition that is isomorphic to $\Gamma$. The equivalence relation on $\R^2$ induced by the action of this group as in \eqref{eq:eq_ind_G} is the same as in  \eqref{eq:latteq}.

We equip $\R^2/\Gamma$ with the quotient topology. Then 
$T^2=\R^2/\Gamma$ is a $2$-dimensional 
torus,\index{torus}\index{Taaa@$T^2$}
and $\pi\: \R^2 \ra 
T^2=\R^2/\Gamma$ is a covering map; actually, since we insist on covering maps being orientation-preserving, we have to equip $T^2$ with a suitable  orientation here. 
The lattice translations $\tau_\ga$, $\ga \in \Gamma$, are deck transformations of 
the quotient map $\pi$ and so $\pi=\pi\circ \tau_{\ga}$ for $\ga\in \Gamma $.
 We will see momentarily that actually every deck transformation of $\pi$ has this form (Lemma~\ref{lem:torilifts}~\ref{item:tori1}).  

 The conformal structure on $\C$ induces a unique conformal
 structure on $\R^2/\Gamma= \C/\Gamma$.  It is represented by a
 complex atlas on $\C/\Gamma$ given by suitable local inverse
 branches of $\pi$.  Then $\T=\C/\Gamma$ is a 
 complex torus\index{complex torus}\index{torus!complex}\index{T@$\T$}
 and
 $\pi \: \C \ra \T$ a holomorphic map. Such a complex torus will
 always be denoted by $\T$, whereas $T^2$ will denote a torus
 that 
 is not equipped with a conformal structure, i.e., a topological
 torus.

Every topological torus 
can be represented in the form $\R^2/\Gamma$ 
up to orientation-preserving homeomorphisms. Actually,  we can  choose $\Gamma=\Z^2$ if this is convenient. Up to conformal equivalence every complex torus has a representation of  the form $\C/\Gamma$. 
Here two complex tori 
$\T=\C/\Gamma$ and $\T'=\C /\Gamma'$ obtained from rank-2 lattices $\Gamma, \Gamma'\sub \C$ are conformally equivalent if and only if there exists $\alpha\in \C\setminus \{0\}$ such that $\Gamma'=\alpha\Gamma$. 

In the following lemma we collect various statements that are used
in Chapter~\ref{cha:lattes-lattes-type}.  

\begin{lemma}\label{lem:torilifts}
Let $\Gamma\sub \R^2$ be a rank-2 lattice, $T^2=\R^2/\Gamma$, and 
$\pi\: \R^2\ra T^2=\R^2/\Gamma$ be the quotient map. 

 \begin{enumerate}
 
   \item
    \label{item:tori1}
 For a  continuous map  $\varphi\: \R^2 \ra \R^2$ we have 
 $\pi\circ \varphi=\pi$ if and only if  there exists $\ga \in \Gamma$ such that $\varphi=\tau_\ga$. 
   
    \item 
    \label{item:tori2}
    If $\overline  A\: T^2\ra T^2$ is a torus endomorphism, then  $\overline  A$ can be lifted to a homeomorphism on $\R^2$, i.e., there exists a homeomorphism $A\: \R^2\ra \R^2$ such that $\overline  A\circ \pi
    =\pi \circ A$. The  homeomorphism $A$ is orientation-preserving, and unique up to postcomposition with a translation $\tau_\ga$,  $\ga\in \Gamma$. 
    
      \item 
    \label{item:tori3} If $\overline  A\: T^2\ra T^2$ is a torus endomorphism, then there exists a unique 
    $\R$-linear map $L\: \R^2\ra \R^2$ with $L(\Gamma)\sub \Gamma$ such that for every lift  $  A$ as in \ref{item:tori2} we have 
    \begin{equation} \label{eq:mapL} 
A\circ \tau_\ga\circ A^{-1}= \tau_{L(\ga)}= L\circ \tau_\ga \circ L^{-1}
\end{equation}
for all $\ga\in \Gamma $.

  \item 
    \label{item:tori4} If $\overline  A\: T^2\ra T^2$ is a torus endomorphism and $L$
    the map as in  \ref{item:tori3}, then 
    $\deg(\overline A)=\det(L)$. 
 \end{enumerate}
\end{lemma}  

Statement~\ref{item:tori1} implies  that the deck transformations 
of  $\pi$ are precisely the lattice translations $\tau_\ga$, $\ga\in \Gamma$. 

For \ref{item:tori2} recall that a 
{\em torus endomorphism}\index{torus!endomorphism}
$\overline A\: T^2\ra T^2$ is an
orien\-tation-preserving local homeomorphism on a torus (see the introduction of Chapter~\ref{cha:lattes-lattes-type}). 
From the statement  we get a commutative diagram of the form:
   \begin{equation}\label{eq:lifttorusendoz}
    \xymatrix{
      \R^2 \ar[r]^A \ar[d]_\pi & \R^2 \ar[d]^\pi
      \\
      T^2 \ar[r]^{\overline {A}} & T^2\rlap{.}
    }
  \end{equation}
  
The map $L$ in  \ref{item:tori3} can be viewed as the map induced by $\overline A$ on the fundamental group of $T^2=\R^2/\Gamma$. Indeed, if $x_0\in T^2$ and $y_0\coloneqq \overline A(x_0)$, then $\overline A$ induces a map 
 $\overline A_*\: \pi_1(T^2, x_0)\ra \pi_1(T^2,y_0)$ (see the end of
 Section~\ref{sec:covmaps}). Moreover, we have a
 natural isomorphism $\pi_1(T^2, x_0)\cong \Gamma$. Namely, if $[\ell]\in \pi_1(T^2, x_0)$ is an element of the fundamental group represented by a loop $\ell$ based at $x_0$, then we can lift $\ell$ to a path $\tilde \ell$ on $\R^2$ by $\pi$. In general, $\tilde \ell$ is not a loop; if $u_0\in\R^2$ is the initial point of $\tilde \ell$ and $v_0\in \R^2$ the other endpoint of $\tilde \ell$, then $\ga\coloneqq v_0-u_0\in\Gamma$, and one can show that the map $[\ell]\mapsto \ga$ gives  a well-defined group isomorphism $\pi_1(T^2, x_0)\ra \Gamma$. Similarly, we have a natural isomorphism $\pi_1(T^2, y_0)\ra  \Gamma$.  Under these isomorphisms, the map $A_*\: \pi_1(T^2, x_0)\ra \pi_1(T^2,y_0)$ corresponds to the homomorphism $L\: \Gamma\ra\Gamma$. 

To see this, note that in our setting $A\circ \tilde \ell$ is a lift of the image 
loop $\overline A\circ \ell$ representing $A_*([\ell])$.  The path  
$ A\circ \tilde \ell$ has the endpoints $A(u_0)$ and $A(v_0)$. So  
under the isomorphism $\pi_1(T^2, x_0)\ra  \Gamma$ and $\pi_1(T^2, y_0)\ra  \Gamma$ the map $[\ell] \mapsto A_*([\ell])=[\overline A\circ \ell]$ corresponds to the map
\begin{align*} \ga\in \Gamma\mapsto A(v_0)-A(u_0) &=(A\circ \tau_\ga)(u_0)-A(u_0)\\
&= (\tau_{L(\ga)}\circ A)(u_0)-A(u_0)=L(\ga)
\end{align*}
by \eqref{eq:mapL}.

  \begin{proof}[Proof of Lemma~\ref{lem:torilifts}]  
  
  \smallskip 
  \ref{item:tori1} Let $\varphi\: \R^2 \ra \R^2$ be a  continuous
  map with \mbox{$\pi\circ \varphi=$} $\pi$. Then 
 we can consider  $\varphi$ as a lift of $\pi\: \R^2 \ra T^2$ by  the covering map $\pi\: \R^2 \ra T^2$, because we have  the commutative diagram 
  \begin{equation*}
  \xymatrix {
  & \R^2  \ar[d]^{\pi}   \\
    \R^2  \ar[ur]^{\varphi}  \ar[r]^{\pi}  &   T^2\rlap{.} 
  }
\end{equation*}

Fix   $u_0\in \R^2$,   and define $v_0=\varphi(u_0)$. Then $\pi(v_0)=\pi(\varphi(u_0))=\pi(u_0)$.  
 This means that $u_0\sim v_0$ are equivalent with respect to the equivalence relation $\sim$ induced by $\Gamma$. Thus, there exists $\ga\in \Gamma$ such that $v_0=u_0+\ga$. Consider the corresponding lattice translation $\tau_\ga$. Then $\pi\circ \tau_\ga=\pi$, and so $\tau_\ga$ is also a lift of 
 $\pi$ by  the covering map $\pi$. Since $\tau_\ga(u_0)=v_0=\varphi(u_0)$, Lemma~\ref{lem:liftsofcov}~\ref{item:liuniq} implies that 
 $\varphi=\tau_\ga$. 
  
  \smallskip 
  \ref{item:tori2} We can lift the map $\overline A \circ \pi\: \R^2 \ra T^2$ by  the covering map $\pi$ to get a continuous map $A\: \R^2\ra \R^2 $ satisfying 
$\pi \circ A=\overline A \circ \pi$ (Lemma~\ref{lem:liftsofbranched}~\ref{item:Liftex}).


In order to find an inverse $B$ of  $A$ we  want to  reverse the roles of $\pi$ and $\overline A \circ \pi$. 
 Lemma~\ref{lem:2_3_branched}~\ref{item:2_out3_1} 
implies that $\overline A \circ \pi$ is  a branched covering map. Since  $\overline A$ and  $\pi$ are local homeomorphisms, $\overline A \circ \pi$ has no critical points and so it is  a covering map. 

 Fix   $u_0\in \R^2$, and 
define  $v_0= A(u_0)\in \R^2$. 
Then  
$$\pi(v_0)=\pi(A(u_0))=\overline A(\pi(u_0)).$$  So    if we  lift $\pi\: \R^2 \ra T^2$ by  the  covering map  $\overline A \circ \pi$, then  we can find  a continuous map $B\: \R^2 \ra \R^2$ with $B(v_0)=u_0$ and $\pi =\overline A \circ \pi \circ B$; so we  obtain  the commutative diagram 
\begin{equation*}
  \xymatrix{
    u_0 \in \R^2 \ar[dr]_{\overline{A}\circ\pi}\ar@<0.5ex>[rr]^A & 
    & v_0 \in \R^2 \ar[dl]^{\pi} \ar@<0.5ex>[ll]^B
    \\
    & T^2\rlap{.}
    }
\end{equation*}
Then  $(A\circ B)(v_0)=A(u_0)=v_0$   and 
$$\pi \circ (A \circ B)= \overline A \circ \pi \circ B=\pi=\pi \circ \id_{\R^2}.$$
So we conclude  $A\circ B=\id_{\R^2}$  by 
Lemma~\ref{lem:liftsofcov}~\ref{item:liuniq} applied to the covering map $\pi$.
Similarly, $(B\circ A)(u_0)=B(v_0)=u_0$ and  
$$  \overline A \circ \pi \circ  B\circ A = \pi \circ A= \overline A \circ \pi= \overline A \circ \pi \circ \id_{\R^2}, $$   
which implies  $B\circ A=\id_{\R^2}$ by the same lemma applied to the covering map 
 $\overline A \circ \pi $.  It follows that $A$ is a homeomorphism on $\R^2$ with inverse $B$.  
  
Since $\pi$ and $\overline A$ are orientation-preserving local
homeomorphisms,   the relation  $\pi \circ A=\overline A \circ
\pi$ in 
combination with Lemma~\ref{lem:23orient} implies 
that  $A$
is also orientation-preserving.   
These considerations show that  $\overline A$ has  a lift $A$ as desired. 
  
  Suppose $A'\: \R^2\ra \R^2$ is another continuous map with 
  $\pi \circ  A'=\overline A \circ \pi$. Then 
  $$ \pi\circ (A'\circ A^{-1})= \pi\circ  A' \circ B= \overline A \circ \pi \circ B=\pi. $$
  By \ref{item:tori1} 
there exists $\gamma\in \Gamma$ 
such that $A'\circ A^{-1}=\tau_\ga$.
  Then $A'=\tau_\ga\circ A$. This shows that $A$ is unique up to
  postcomposition with a lattice translation. Note that clearly
  every map of the form $\tau_\ga\circ A$, 
$\gamma\in \Gamma$, is actually 
a lift 
  of $\overline A$.

  \smallskip
  \ref{item:tori3} Let $A$ be a homeomorphic lift  of $\overline A$ as in \ref{item:tori2}. If $\ga\in \Gamma$ is arbitrary, 
then
\begin{align*}  \pi \circ A \circ \tau_\ga  \circ A^{-1}&= \overline A \circ \pi  \circ \tau_\ga  \circ A^{-1} \\ & = 
\overline A \circ \pi   \circ A^{-1}=\pi \circ A   \circ A^{-1}=\pi. 
\end{align*} 
So $ A \circ \tau_\ga  \circ A^{-1}$ is a deck transformation of
the covering map $\pi$. By \ref{item:tori1} this implies that there exists a unique $\ga'\in \Gamma$ such that 
$$ A \circ \tau_\ga  \circ A^{-1}= \tau_{\ga'}. $$

Note that for $\gamma, \sigma\in \Gamma$ we have
\begin{align*}
A\circ \tau_{\gamma + \sigma} \circ A^{-1} &= A \circ \tau_\gamma \circ
\tau_\sigma \circ A^{-1} = \tau_{\gamma'} \circ A \circ \tau_\sigma \circ
A^{-1} \\ &= \tau_{\gamma'} \circ \tau_{\sigma'} = \tau_{\gamma' + \sigma'}.
\end{align*} 
Moreover, if 
$$\tau_{\ga'}=A\circ \tau_\gamma \circ A^{-1} =\tau_0=\id_{\R^2}, $$ then 
$$\tau_{\gamma}=A^{-1}\circ \id_{\R^2}\circ A=\id_{\R^2}, $$ and so $\ga=0$. 
This implies that  the map $\ga\mapsto \ga'$  gives an injective group homomorphism $L\: \Gamma \ra  \Gamma $ such that $A\circ \tau_\ga\circ A^{-1}= \tau_{L(\ga)}$ for all $\ga \in \Gamma$. 
We can uniquely extend $L$ to an invertible $\R$-linear map 
on $\R^2$, also denoted by  $L$. Then $L(\Gamma)\sub \Gamma$ and the first equality in \eqref{eq:mapL} is true by definition of $L$. The second equation in  \eqref{eq:mapL} is actually true for {\em every} invertible linear map $L$.  Uniqueness of $L$ is clear. 

  \smallskip
  \ref{item:tori4} Let $A\:\R^2\ra \R^2$  be a lift of $\overline A$ as in 
  \ref{item:tori2} and $L$ be the linear map as in \ref{item:tori3}.  
  We define a homotopy $H\: \R^2\times [0,1]\ra \R^2$ as 
$$H(u,t)= (1-t)A(u)+tL(u), \quad u\in \R^2, \, t\in [0,1]. $$  
 Let $H_t\coloneqq  H(\cdot, t)$ for $t\in [0,1]$. Then $H_0=A$ and $H_1=L$. 
For all $u\in \R^2$, $t\in [0,1]$,  and $\ga \in \Gamma$ we have 
\begin{align*} (H_t\circ \tau_\ga)(u)&= (1-t)A(\tau_\ga(u))+t L(\tau_\ga(u))\\
& =(1-t)\tau_{L(\ga)}(A(u))+ t 
\tau_{L(\ga)}(L(u))\\
&=(1-t)A(u)+tL(u)+L(\ga)=(\tau_{L(\ga)}\circ H_t)(u). 
\end{align*} 
Here we used \eqref{eq:mapL}. 
Hence $H_t\circ \tau_\ga= \tau_{L(\ga)}\circ H_t$ for all $t\in
[0,1]$ and $\ga \in \Gamma$. This implies that the equivalence
relation $\sim$ induced by $\Gamma$ is invariant under $H_t$, and so $H_t$ descends to a continuous map $\overline H_t$ on 
$T^2=\R^2/\Gamma$ (see Lemma~\ref{lem:f_descends}). Then
$\pi \circ H_t = \overline H_t \circ \pi$ for all $t\in [0,1]$. In particular,
$$ \overline A \circ \pi =\pi \circ A= \pi\circ H_0 =\overline H_0\circ \pi, $$  which implies $\overline H_0=\overline A$. 

Define $\overline H \: T^2 \times [0,1]\ra T^2$ as $\overline H(x,t)=\overline H_t(x)$
for $(x,t)\in T^2 \times [0,1]$.  Then $\overline H$ is continuous.  Indeed, if 
$\{(x_n,t_n)\}$ is a sequence in $T^2\times [0,1]$ with $(x_n,t_n)\ra (x,t)\in T^2\times 
[0,1]$ as $n\to \infty$, then by considering a local inverse of
$\pi$ near $(x,t)$, we can find points $u\in \R^2$ and $u_n\in
\R^2$ for $n\in \N$ 
with $\pi(u)=x$, $\pi(u_n)=x_n$, and $u_n\to u$ as $n\to \infty$. 
Then as $n\to \infty$, 
\begin{align*}
\overline H(x_n,t_n)= \overline H_{t_n}(\pi(u_n))&=\pi( H_{t_n}(u_n))=\pi(H(u_n,t_n))\ra \pi(H(u,t))\\&= \pi(H_t(u))=\overline H_{t}(\pi(u))=\overline H(x,t)
\end{align*}  by the continuity of $H$ and $\pi$. 

The map   $\overline H$ is a homotopy on $T^2$ with the time-$t$ maps $\overline H_t$. In particular, the maps  $\overline A=\overline H_0$ and $\overline L\coloneqq \overline H_1$ on $T^2$ are homotopic.  Since degrees of maps are invariant under homotopies, it follows that $ \deg(\overline A)=\deg(\overline L)$. 
Note that 
$ \overline L\circ \pi =\overline H_1\circ \pi= \pi \circ H_1= \pi  \circ L$. 
So  $\overline L$ is a  map on $T^2$ induced by the linear map $L$. 

It   is a standard fact that 
then $\deg(\overline L)= \det(L)$. 
One can see this as follows by  using some basic concepts from differential geometry. Namely, we can consider $\T= \R^2/\Gamma$ as a complex torus and hence as a smooth manifold.  Then $\overline L$ is a smooth map on $\T$.
There exists a unique $2$-form   $\alpha$  on $\T$ 
whose  pull-back $\pi^*(\alpha)$ 
 is equal to the standard volume form on $\R^2$ given by $\om=dx\wedge  dy$, where  $x$ and $y$ are the standard coordinates on $\R^2$.  If we pull $\omega$ back by the linear map $L$, we  get   
$L^*(\omega)=\det(L) \omega$. This implies that $\overline L^*(\alpha) =\det(L)\alpha$. On the other hand, 
$$ \int_\T\overline L^*(\alpha)=\deg (\overline L)\int_\T\alpha. $$ 
Hence 
$$ \det(L) \int_\T \alpha  =  \int_\T\overline L^*(\alpha)=\deg (\overline L)\int_\T \alpha. $$
Since $\displaystyle \int_\T\alpha \ne 0$ this implies that $\deg(\overline L)=\det(L)$. 
We conclude that 
$$ \deg(\overline A)= \deg(\overline L) =\det(L)$$ as desired. 
\end{proof}

%

In Lemma~\ref{lem:torilifts} we can consider a complex torus $\T=\C/\Gamma$.
Then the quotient map $\pi\: \C \ra \T=\C/\Gamma$ is holomorphic. Suppose 
 $\overline A\: \T\ra \T$ is a holomorphic torus endomorphism, and $A\:\C \ra \C$ is a homeomorphic lift of $\overline A$ as in  
 Lemma~\ref{lem:torilifts}~\ref{item:tori2}. Then $\pi\circ A=\overline A \circ \pi$. This shows that locally the map $A$ is given as a composition of the holomorphic map $\overline A \circ \pi$ with a branch of $\pi^{-1}$,  which is also holomorphic. We conclude that   $A\: \C \ra \C$ is holomorphic itself (we can also apply 
Lemma~\ref{lem:2_3_branched} here). Since this map is also a homeomorphism on $\C$  it  must be of the form $A(z)=\alpha z+\beta$
for $z\in \C$, where $\alpha,\beta\in \C$, $\alpha\ne 0$.   
The associated linear map $L\: \C \ra \C$ as in Lemma~\ref{lem:torilifts}~\ref{item:tori3} is then given by $L(z)=\alpha z$
for $z\in \C$. 
By part \ref{item:tori4} of this lemma  we  have $\deg(\overline A)=\det(L)=|\alpha|^2$ (note that here  we have to consider $L$ as an $\R$-linear map).



\section{Orbifolds and coverings}
\label{sec:orbifolds-coverings}

In this section we will discuss the relation of orbifolds and
branched covering maps. Recall (see Section~\ref{sec:orbif-assoc-thurst}) that 
 an 
{\em orbifold}\index{orbifold} is a pair $\mathcal{O} =(S,\alpha)$,
where $S$ is a surface and $\alpha\colon S\to \widehat{\N}=\N\cup\{\infty\}$ 
a ramification function on $S$. 

We will restrict ourselves to the  case
relevant for Thurston maps, where the underlying surface of the orbifold $\mathcal{O} =(S,\alpha)$
 is an oriented  topological $2$-sphere $S=S^2$.  Then   $\supp(\alpha)=\{ p\in S^2: \alpha(p)\ge 2\}$
    is a finite set. 
      A point $p\in \supp(\alpha)$ is called a  
      \emph{conical singularity}\index{conical singularity}\index{singularity!conical} or 
\emph{cone point}\index{cone!point}\index{orbifold!cone point
  of}  
if 
 $2\le\alpha(p)<\infty$,  and a 
{\em puncture}\index{puncture}\index{orbifold!puncture of}    
if
  $\alpha(p)=\infty$. 
  
  Often the underlying sphere $S^2$ of the orbifold 
  will be the Riemann sphere
$\CDach$.  In this case, the orbifold $(\CDach, \alpha)$
has an underlying  conformal structure,
whereas in general we consider $(S^2,\alpha)$ as a purely topological object.
  
  Recall (see Section~\ref{sec:orbif-assoc-thurst}) that 
the 
\emph{Euler characteristic}\index{orbifold!Euler characteristic|textbf}\index{Euler characteristic!of orbifold|textbf}
of an orbifold $\OC=(S^2,\alpha)$ is defined as 
\begin{equation}
  \label{eq:def_Euler_orbi}
  \chi(\OC)= 2 - \sum_{p\in S^2}\left(1- \frac{1}{\alpha(p)}\right),
\end{equation}
where  we use the convention $1/\infty=0$. The orbifold   $\OC$
is 
\emph{parabolic}\index{parabolic!orbifold}\index{orbifold!parabolic}
 if 
 $\chi(\OC)=0$ and 
\emph{hyperbolic}\index{hyperbolic!orbifold}\index{orbifold!hyperbolic} 
if $\chi(\OC)< 0$

One can give a geometric  interpretation of $\chi(\OC)$ also as follows.   Let $\DD$ be
a cell decomposition of $S^2$ (see Section~\ref{s:celldecomp}) such that
each point $p\in \supp(\alpha)$  is a vertex of $\DD$. 
We count each vertex  $p$  in $\DD$ with weight   $1/\alpha(p)$, and
 define
\begin{equation}
  \label{eq:orbi_V}
  \#V= \sum\frac{1}{\alpha(p)},
\end{equation}
 where the sum is taken over all vertices $p$ in $\DD$. Then it easily follows from Euler's polyhedral formula that 
\begin{equation*}
  \chi(\OC)= \#F-\#E+\#V,
\end{equation*}
where $\#F$ is the number of faces (i.e., $2$-cells) and $\#E$ is
the number of edges (i.e., $1$-cells) in  $\DD$. So with the slight
modification~\eqref{eq:orbi_V} the 
{Euler characteristic}
 of an orbifold
is given by the usual formula. 
Note that punctures do not contribute to
$\#V$ in \eqref{eq:orbi_V}.  

We will sometimes remove the punctures of  the orbifold 
$(S^2, \alpha)$ from the $2$-sphere. 
With the ramification function $\alpha$ understood, we set 
\index{S20@$S^2_0$}
\begin{equation}
\label{eq:rem_punctures}
  S^2_0=S^2\setminus\{p\in S^2 : \alpha(p)=\infty\}. 
\end{equation}
In particular, we use this notation if  $S^2=\CDach$ and 
so $\CDach_0$ denotes the Riemann sphere without 
 the punctures of the given orbifold.

The following  statement  relates orbifolds and branched covering maps. 

\begin{theorem}
  \label{thm:univ_orbi_cover}
  Let $\OC=(S^2,\alpha)$ be an orbifold that is parabolic or
  hyperbolic. Then the following statements are true:
  \begin{enumerate}
  \item 
    \label{item:ex_univ_orbi1}
    There exist a simply connected surface $X$ and a branched
    covering map $\Theta\colon X\to S^2_0$ such that
    \begin{equation*}
      \deg(\Theta, x)= \alpha(\Theta(x))
    \end{equation*}
    for each  $x\in X$. 
    
  \item \label{item:ex_univ_orbi2} If $S^2=\CDach $,   then in 
  \ref{item:ex_univ_orbi1} we may in addition assume that   $X=\C$ if 
  $\mathcal{O}$ is parabolic,  $X=\D$ if  $\mathcal{O}$ is hyperbolic,  and that the map $\Theta\: X\ra \CDach_0$ is holomorphic.
  \end{enumerate}
\end{theorem}

We will not provide a proof here. The statement follows  from
\cite[Theorem~1]{Bers}. 
See also  \cite[4.8.2]{Lam},
\cite[Proposition~13.2.4 and Theorem~13.3.6]{Th}, and (for 
a much more general situation)  \cite{BH}.

Part~\ref{item:ex_univ_orbi1} of the previous theorem 
implies that $\deg(\Theta,x) = \deg(\Theta,y)$ for
all $x,y\in X$ with $\Theta(x)=\Theta(y)$, i.e., the local degree of
$\Theta$ is constant in each fiber $\Theta^{-1}(p)$, $p\in S_0^2$. 
Note that $\Theta$ does not cover the punctures of $\OC$.

Theorem~\ref{thm:univ_orbi_cover}~\ref{item:ex_univ_orbi2}
explains why we  call orbifolds  parabolic or  hyperbolic. Here the additional conformal structures are important, because  $\C$ and $\D$ are topologically   indistinguishable.


\begin{definition}
  \label{def:univ_orbi_cover}
  For a  parabolic  or  hyperbolic orbifold $\OC=(S^2,\alpha)$ the map
  $\Theta\colon X\to S^2_0$ from Theorem~\ref{thm:univ_orbi_cover} is
  called the 
  \emph{universal orbifold covering map}\index{universal orbifold covering map}\index{orbifold!universal covering map} 
  of $\OC$. 
\end{definition}

We will momentarily see that up to equivalence the map $\Theta$ is uniquely determined. First we formulate the  universal 
 property of $\Theta$ (for  a closely related statement see  \cite[Proposition~13.2.4]{Th}). 

\begin{theorem}
  \label{thm:unv_orbi_prop}
  Let $\OC=(S^2,\alpha)$ be a parabolic or hyperbolic orbifold with 
  universal orbifold covering map $\Theta\: X\ra S^2_0$,  
$Z$ be a surface, and $f\: Z\ra S^2_0$ be a branched covering map 
such that $$\deg(f,z) |\alpha (f(z))$$ for each $z\in Z$.

Then for all  points  $x_0\in X$ and $z_0\in Z$  with $p_0\coloneqq \Theta(x_0)=f(z_0)$ 
there exists a branched covering map $\varphi\: X\ra Z$ such that $\varphi(x_0)=z_0$ and 
$\Theta = f\circ \varphi$. Moreover, if $\alpha(p_0)=1$, then the map $\varphi$ with these properties is unique. 
  \end{theorem}

So the following diagram commutes: 
\begin{equation*}
  \xymatrix{
    & X \ar[d]^\Theta \ar[dl]_\varphi
    \\
    Z  \ar[r]^f& S_0^2\rlap{.}
    }
\end{equation*}


%

\begin{proof} Note that if $x\in X$, $z\in Z$, and $p\coloneqq \Theta(x)=f(z)$, 
then 
$$ \deg(f,z)|\alpha(p)=\deg(\Theta, x). $$
It follows that we may lift $\Theta$ by $f$ to yield the desired
map $\varphi$ according to
Lemma \ref{lem:liftsofbranched}~\ref{item:Liftex} (note that $f$
plays the role of $\pi$ in this lemma).

If in addition $\alpha(p_0)=1$, then $\deg(f,z)=1$ for each point $z\in f^{-1}(p_0)$, and so the fiber $f^{-1}(p_0)$ is clean. Hence 
 Lemma~\ref{lem:liftsofbranched}~\ref{item:Liftuniq} implies that with the stated properties the map
 $\varphi$ is uniquely determined. 
\end{proof} 

\begin{cor}[Uniqueness of the universal orbifold cover]
  \label{cor:unique_univ_orbi}
  Let $\OC=(S^2,\alpha)$ be an orbifold that is parabolic 
or hyperbolic, and 
$\Theta\colon X\to S_0^2$ and 
  $\widetilde \Theta\colon\widetilde {X}\to S_0^2$ be universal orbifold
  covering maps. Then for all points $x_0\in X$ and $\widetilde{x}_0\in \widetilde X$
  with $p_0\coloneqq \Theta(x_0)=\widetilde \Theta(\widetilde{x}_0)$ there exists an orientation-preserving homeomorphism $A\: X \ra \widetilde X$ 
  with $A(x_0)=\widetilde{x}_0$ and $\Theta =\widetilde \Theta \circ
  A$. Moreover, if $\alpha(p_0)=1$, then $A$ is unique.
 \end{cor}
 
So  the universal orbifold covering map of $\OC$ is unique 
 up to precomposition with an orientation-preserving  homeomorphism.  

\begin{proof}
We  can apply  Theorem~\ref{thm:unv_orbi_prop}  to the universal 
covering map  $\Theta\colon X\to
  S^2_0$ and the branched covering map $\widetilde 
  \Theta\colon \widetilde X\to S_0^2$. This gives the existence of a branched covering map $A\: X\ra \widetilde X$ with $A(x_0)=\widetilde{x}_0$ and $\Theta=\widetilde \Theta\circ A$. 
 
 To show that $A$ is a homeomorphism, we have to construct an inverse for $A$. For this we want to reverse the roles of $\Theta$ and $\widetilde \Theta$, and apply the uniqueness statement in Theorem~\ref{thm:unv_orbi_prop}. Here a  complication is that the fibers of the maps above $p_0$ are not clean if $\alpha(p_0)\ge 2$. So we choose a point $p_1\in S^2$
 with $\alpha(p_1)=1$,   a point  
 $x_1\in \Theta^{-1}(p_1)$,  and set 
 $\widetilde x_1=A(x_1)$. Then  again by Theorem~\ref{thm:unv_orbi_prop}
 there exists a branched covering map $B\: \widetilde X \ra X$ with 
 $B(\widetilde{x}_1)= x_1$ and $ \widetilde \Theta=\Theta\circ B$. 
 
  Then  $\widetilde \Theta=\widetilde \Theta \circ A\circ B$ and 
   $\Theta= \Theta\circ B\circ A$. These relations and Lemma~\ref{lem:2_3_branched}~\ref{item:2_out_3_2} imply that $A\circ B\: \widetilde X \ra \widetilde X $ and $B\circ A\: X\ra X $ are branched covering maps.

   Since $(A\circ B)(\widetilde{x}_1)=\widetilde{x}_1$,  $(B\circ A)(x_1)=x_1$, and 
   $$ \alpha(\Theta(x_1))=\alpha(\widetilde \Theta(\widetilde{x}_1))=\alpha(p_1)=1,$$
   it follows from the uniqueness statement in Theorem~\ref{thm:unv_orbi_prop} that 
  $A\circ B= \id_{\widetilde X}$ and  $B\circ A=\id_{X}$. 
  Therefore, $A$ is a
  homeomorphism. As a branched covering map,  $A$ is necessarily orientation-preserving.  
  
  If $\alpha(p_0)=1$, then the uniqueness of $A$ again follows  from 
  Theorem~\ref{thm:unv_orbi_prop}.
\end{proof}

\begin{rem} \label{rem:holoorbset} Theorem~\ref{thm:unv_orbi_prop} remains valid in the holomorphic setting with the obvious 
 changes in the formulation. Namely, suppose that in addition to the hypotheses in the statement 
 $S^2=\CDach$, $X$ and $Z$ are Riemann surfaces, and $\Theta$ and $f$ are holomorphic.  
 Then the map $\varphi$ with $\Theta=f\circ \varphi$ is also holomorphic. This follows from 
 the last part of Lemma~\ref{lem:2_3_branched}.
 
 Similarly,  in the holomorphic setting in  Corollary~\ref{cor:unique_univ_orbi} the map $A$ will be holomorphic and hence a biholomorphism. In particular, if $\Theta$ is the  universal orbifold covering map as in  
 Theorem~\ref{thm:univ_orbi_cover}~\ref{item:ex_univ_orbi2}, then $\Theta$ is unique up to precomposition with a suitable biholomorphism. 
\end{rem}

Let $\OC=(S^2,\alpha)$ be a parabolic or hyperbolic orbifold,  
  and $\Theta\colon X\to S_0^2$  be the universal orbifold
  covering map. 
A  homeomorphism $g\colon X\to X$ is called a
\emph{deck transformation}\index{deck transformation}\index{universal orbifold covering map!deck transformation}\index{orbifold!universal covering map!deck transformation}
 of $\Theta$  if $ \Theta\circ g = \Theta$.
This relation and Lemma~\ref{lem:23orient} imply that  each deck transformation is orientation-preserving. 

The deck transformations of $\Theta$ 
form a group under composition, denoted by  $\pi_1(\OC)$ and called the
\emph{fundamental group}\index{orbifold!fundamental group}\index{fundamental group of orbifold}
  of the orbifold $\OC$. 
  

We  collect properties of the group of deck
transformations in the following statement. 

\begin{prop}[Deck transformations of the universal orbifold cover]
  \label{prop:prop_deck_trafo}
  Let $\OC=(S^2,\alpha)$ be an orbifold that is parabolic   or hyperbolic,
  and $\Theta\colon X\to \OC$ be the universal orbifold covering
  map. Then for the group of deck transformations $G= \pi_1(\OC)$ the following statements are true: 

  \begin{enumerate}
  \item 
    \label{item:pi1_O_1}
    The map $\Theta$ is induced by $G$, meaning that for all $x,y\in X$ we have 
  \begin{equation*}
          \Theta(x)=\Theta(y) \text{ if and only if there exists $g\in G$ with } y=g(x). 
    \end{equation*}
      \item 
    \label{item:pi1_O_2}
   For all $x\in X$ the stabilizer 
$G_x=\{g\in G: g(x)=x\}$ is a finite cyclic group of order 
 \begin{equation*}
    \#G_x=  \deg(\Theta,x)=\alpha(\Theta(x)). 
    \end{equation*}

    \item 
    \label{item:pi1_O_3}
    $G$ acts properly discontinuously on $X$.
    
     \item
    \label{item:pi1_O_4}
   If $\mathcal{O}$ has no punctures, then $G$ acts cocompactly on $X$.

    \setcounter{mylistnum}{\value{enumi}}
  \end{enumerate}
  
 Assume in addition that $S^2=\CDach$,  that 
   $X=\C$ or $X=\D$ depending on whether  $\mathcal{O}$ is parabolic or hyperbolic, and that $\Theta\: X\ra \CDach$ is holomorphic.
   \begin{enumerate}
    \setcounter{enumi}{\value{mylistnum}}
  \item 
    \label{item:pi1_O_5} Each $g\in G$ acts as a biholomorphism on $X$. Moreover, 
    if we equip $X=\C$ with the Euclidean metric and $X=\D$ with the hyperbolic metric, then each $g\in G$ acts on $X$ as an isometry.
  \setcounter{mylistnum}{\value{enumi}}
  \end{enumerate}
 \end{prop}
 
Note in particular that by \ref{item:pi1_O_1} the branched
covering map $\Theta$ is 
{\em regular}\index{branched covering map!regular}\index{regular branched covering map} 
in the sense that its group  $G$ of deck transformations  acts transitively on the fibers of $\Theta$.

\begin{proof} It is easiest to reduce  to the holomorphic case. For this 
we choose a conformal structure on $S^2$, represented by an atlas of holomorphically compatible charts. With this conformal structure, $S^2$ can be identified  with $\CDach$ and we can make the additional assumptions  stated before \ref{item:pi1_O_5} without loss of generality. In this reduction it is 
important that all the desired conclusions in 
\ref{item:pi1_O_1}--\ref{item:pi1_O_4} are independent of the choice the universal orbifold covering map $\Theta$ of $\mathcal{O}$. This  follows from the essential uniqueness of $\Theta$ as formulated in Corollary~\ref{cor:unique_univ_orbi} and allows us to switch to a holomorphic version of $\Theta$ defined on $X=\C$ or $X=\D$.  

If $g\in G$, then   $g$ is  an orientation-preserving homeomorphism on $X$. Since 
$\Theta=\Theta\circ g$, it follows from 
the last  part of Lemma~\ref{lem:2_3_branched} that 
 $g$  is holomorphic and hence a biholomorphism on $X=\C$ or $X=\D$. 

\smallskip 
\ref{item:pi1_O_1} This immediately follows from the definition of  a deck transformation and Corollary~\ref{cor:unique_univ_orbi}. 

In particular, the orbits of points under $G$ are precisely the fibers of the map $\Theta$. Since $\Theta$ is  a branched covering map, these fibers, and hence the orbits of points under $G$, are discrete sets in $X$. 

\smallskip
\ref{item:pi1_O_5} We know that if $g\in G$, then $g$ acts as a biholomorphism on $X$. So if $X=\D$, then $g$ is a M\"obius transformation with $g(\D)=\D$. 
Hence $g$ preserves the hyperbolic metric $d_0$ on $\D$ and  acts as an isometry on $X$.

If $X=\C$, then we can say that $g(z)=a z+ b$ for $z\in \C$, where $a, b\in \C$, $a \ne 0$. Here $|a|=1$, and so $g$ acts as an isometry on $\C$ equipped with the Euclidean metric; indeed, otherwise $g$ or $g^{-1}$ has an attracting fixed point and this would produce orbits under $G$ that are not discrete in $X$. This contradicts what we have seen in the proof of \ref{item:pi1_O_1}.

\smallskip
\ref{item:pi1_O_2} We may assume that $x=0$. Then an element $g\in G$ in the stabilizer $G_x$ is necessarily a rotation of $X=\C$ or $X=\D$ 
around $0$. Since $G$ has discrete orbits, this implies that $G_x$ is a 
finite cyclic group of rotations. 

Two points $u$ and $v$ near $x$ lie in the same orbit   of $G$ if and only if there exists $g\in G_x$ such that $v=g(u)$. Indeed, since orbits of $G$ are discrete in $X$, an element $g\in G\setminus G_x$ moves $x$ and hence also the nearby point $u$ a definite distance away from $x$. Hence $g$ cannot map $u$ to the point $v$ near $x$.

If $d=\deg(\Theta, x)=\alpha(\Theta(x))$, then $\Theta$ is $d$-to-$1$ near $x$; so for each 
point $x'\ne x$ near $x$ there are precisely $d$ points near $x$ that are mapped 
to the same image $\Theta(x')$, or, equivalently, lie in the $G$-orbit of $x'$. By what we have seen, these $d$ elements must be 
the points  in the  $G_x$-orbit of $x'$. Since $x'\ne x$, there are precisely $\#G_x$ such points, and we conclude 
$\#G_x =d$ as desired.

\smallskip
\ref{item:pi1_O_3} We argue by contradiction and assume that there exist a compact set $K\sub X$ and pairwise distinct elements $g_n\in G$ for $n\in \N$ with $g_n(K) \cap K\ne \emptyset$. Then we can find $x_n\in K$ such that $y_n=g_n(x_n)\in K$. By passing to subsequences if necessary, we may assume $x_n\to x$ and $y_n\to y$ as $n\to \infty$, where $x,y\in K$. Then $g_n(x)=y$ for all $n$ large. Indeed, if $g_n(x)\ne y$, then by the discreteness of orbits $g_n$ sends $x$ and hence nearby points a definite distance away from $y$. This is impossible  for large $n$, because $g_n$ sends the point $x_n$ near $x$ to the point $y_n$ near $y$. 

By  \ref{item:pi1_O_2} the stabilizer group  $G_x$ of $x$ is finite. 
On the other hand, by what we have just seen, if $n$ is large enough, $G_x$ contains the infinitely many distinct  elements $g_k^{-1}\circ g_n$, $k\ge n$. This is a contradiction.  

\smallskip
\ref{item:pi1_O_4} If $\mathcal{O}$ has no punctures, then $S_0^2=S^2=\CDach$ and so $\Theta\: X\ra 
\CDach$ is a branched covering map with target $\CDach$. In
particular, each point $p\in \CDach$ has an open neighborhood
$V\subset \CDach$
that is evenly covered by 
$\Theta$. By shrinking $V$ if necessary (see Lemma~\ref{lem:evennei}), we may assume that the connected components  of $\Theta^{-1}(V)$ are compactly contained in $X$.  In particular, we can find a set $U\sub X$ with $\Theta(U)=V$ such that $\overline U$ is compact. Finitely many of the sets $V$ will cover $\CDach$. If we let $K$ be the closure of the finitely many corresponding sets $U$, then $K\sub X $ is compact and 
$\Theta(K)=\CDach$.  This implies that each  fiber of $\Theta$ and hence 
also each orbit of $G$ contains a point in $K$. So the sets $g(K)$, $ g\in G$, cover $X$.  The cocompactness of $G$ follows. 
 \end{proof}

We now consider the orbifold $\OC_f=(S^2, \alpha_f)$ of a Thurston
map $f\colon S^2 \to S^2$. As the following statement shows, we  can lift   branches of $f^{-1}$
to the universal orbifold cover of $\OC_f$.

\begin{lemma}
  [Existence of lifts to the universal orbifold cover]
  \label{lem:lift_Thurst_orbi}\index{universal orbifold covering map!lift by}\index{lift!by universal orbifold covering map}\index{lift!inverse branch} 
  Let $f\: S^2 \ra S^2$ be a Thurston map,  $\OC_f=(S^2, \alpha_f)$ be
  the associated orbifold,  $\Theta\: X \ra S^2_0$ be the universal
   orbifold covering map, and   
  $z_0,w_0\in X$ be points with  $(f\circ
  \Theta)(w_0)=\Theta(z_0)$. 
  
  Then there exists a branched covering map $A\: X \ra X$ with $
    A(z_0)=w_0$ and 
    $$f\circ \Theta \circ A=\Theta.$$
 If $z_0\notin \crit(\Theta)$, then  the map $A$ is unique.
 
  If we assume that $S^2=\CDach$, $X$ is a Riemann surface, 
   and the maps  $f$ and $\Theta$ are holomorphic, then  $A$ is holomorphic as well.
\end{lemma}
So in this setting, we have the following commutative diagram:
\begin{equation}\label{eq:liftinverse}
  \xymatrix{
    w_0\in   X \ar[d]_\Theta & \ar[l]_A z_0\in  X \ar[d]^\Theta
    \\
    S^2 \ar[r]^{f} & S^2\rlap{.}
  }
\end{equation}
One should think of $A$ as a lift of a suitable inverse branch of $f^{-1}$ by 
the branched covering map $\Theta$. 

\begin{proof} 
  We want to apply Theorem~\ref{thm:unv_orbi_prop} for the
   map 
$f\circ \Theta\: X\ra S^2$, which is a branched covering map by
 Lemma~\ref{lem:2_3_branched}~\ref{item:2_out3_1}.    To see
  that the hypotheses of Theorem~\ref{thm:unv_orbi_prop} are satisfied, let $z, w\in X$
  with $\Theta(z)= (f\circ \Theta)(w)$ be arbitrary. Define
  $p=\Theta(z)$ and $q=\Theta(w)$. Then $p=f(q)$, and so
  $\deg(f,q)\alpha_f(q)$ divides $\alpha_f(p)$ (see
  Proposition~\ref{prop:weightf}~\ref{item:weight1}). Since 
  $\deg(\Theta,w)=\alpha_f(q)$ and $\deg(\Theta,z)=\alpha_f(p)$ by
  definition of the universal orbifold covering map $\Theta$, this implies that
  \begin{equation*}
    \deg(f,q)\alpha_f(q)=\deg(f,q)\deg(\Theta,w)=\deg(f\circ\Theta,w) 
  \end{equation*}
  divides $\deg(\Theta,z)=\alpha_f(p)$. So by 
  Theorem~\ref{thm:unv_orbi_prop} there exists a
  branched covering map $A\: X \ra X$ with $A(z_0)=w_0$ and $f\circ
  \Theta \circ A=\Theta$.
 
 If  $z_0\notin\crit(\Theta)$, then $\deg(\Theta, z_0)=\alpha(\Theta(z_0))=1$. The uniqueness of $A$ then follows from the uniqueness statement in Theorem~\ref{thm:unv_orbi_prop}.

Finally, if  $f$ and  $\Theta$ are holomorphic, then $f\circ \Theta$ is  a holomorphic branched covering map. Since $(f\circ \Theta)\circ A=f$,
the holomorphicity of $A$ follows from Lemma~\ref{lem:2_3_branched}.  
\end{proof}

\section{The canonical orbifold metric}
\label{sec:expratThmaps}

In this section we discuss how to obtain an associated geometric structure from  orbifold data
given by a ramification function on a Riemann surface $S$. This
geometric structure is represented by a metric on $S$, the {\em
  canonical orbifold metric}, and a measure, the {\em canonical
  orbifold measure}.  We will restrict ourselves to the only case
that is  important for us, namely when  $S$ is the Riemann sphere
$\CDach$. The underlying conformal structure on $\CDach$ is
important here, because in this case the universal orbifold cover
carries a 
natural metric and measure 
that we will push forward to the orbifold. 
We first  discuss some related geometric facts that will be relevant 
for understanding the local geometry of an orbifold. 
 
 A \emph{sector} is  a set of the form
\begin{equation*}
  \Sigma= \{ re^{\iu\theta} : 0\le r< r_0, \, 0 \leq \theta \leq \theta_0\},  
\end{equation*}
where  $r_0> 0$ and  $\theta_0 \in (0, 2\pi]$. If we identify
the points $r$ and $re^{\iu \theta_0}$ on the boundary of $\Sigma$
for $0\le r <r_0$,  we obtain a 
\emph{cone}\index{cone} 
$C$ 
of \emph{cone angle} 
$\theta_0$. The  point $0$ (or rather its image in the quotient $C$) is
called the {\em conical singularity} or 
\emph{cone point}\index{cone!point}\index{orbifold!cone point of}\index{conical singularity}\index{singularity!conical}
 of the cone. For the moment  we only allow cones  angles $0<\theta_0\le2\pi$ (see below for the general case).  

A  cone $C$ carries a natural length metric  induced by the Euclidean metric on $\Sigma$. If 
$r_0\le 1$, then $\Sigma\sub \D$ and one can equip $\Sigma$ also with the hyperbolic metric on $\D$, and 
push it to a length metric  on $C$. Depending on this choice of the metric, one calls 
$C$ a {\em Euclidean} or {\em hyperbolic cone}. If we denote this metric on $C$  by $\om$ in both cases
and by $d_0$ the Euclidean or hyperbolic metric on $\Sigma$, then 
the quotient map $(\Sigma,d_0)\ra (C,\om)$ is a {\em path isometry} in that it preserves lengths of paths. 
This property uniquely determines $\om$. 

If we set $\alpha=2\pi/\theta_0\ge 1$, then the map $w\mapsto w^{\alpha}$ (with the branch chosen to be positive on the positive real axis) induces a well-defined homeomorphism of $C$ onto the Euclidean disk $D=B_\C(0, r_0^{\alpha})$. This defines a chart  giving $C$ a Riemann surface structure. 

By using this  chart  one can 
push\index{push-forward!of metric!by orbifold covering map}  
the metric  on 
$C$ to $D$. Then an easy computation shows that $C$ is 
 isometric to the disk $D$ equipped with a (singular) conformal metric, where the length element on $D$ is 
\begin{equation}\label{eq:conemetr1} \frac{|dz|} {\alpha |z|^{1-1/\alpha}}
\end{equation}
or 
\begin{equation} \label{eq:conemetr2}
\frac{2 \, |dz|}  {\alpha  |z|^{1-1/\alpha} (1- |z|^{2/\alpha})} 
\end{equation} in the Euclidean or 
in the hyperbolic case, respectively.  
So the cone $C$ can be viewed in two different ways: as a Riemann surface biholomorphic to a disk  or as a smooth manifold  equipped with 
a  Riemannian metric as in \eqref{eq:conemetr1} or  \eqref{eq:conemetr2}  with a singularity at the cone point. 

Note that the conformal metrics in \eqref{eq:conemetr1} and 
\eqref{eq:conemetr2} are defined for each parameter $\alpha>0$. One can use 
this to define  Euclidean or hyperbolic cones $C$ for arbitrary cone angles $\theta_0>0$. Namely, such a cone $C$  is given by a disk $D$ as above equipped with the  conformal metric \eqref{eq:conemetr1}  
for Euclidean cones, and the conformal metric  \eqref{eq:conemetr2}   for   hyperbolic cones, where
 $\alpha=2\pi/\theta_0$.

In many   cases that are relevant for  us, the cone angle has 
 the form $\theta_0=
2\pi/n$ for some $n\in \N$.  Then one can obtain $C$ also as a quotient 
under a group action. For this we consider the cyclic  group $G$ 
of order $n$ generated by the rotation $z\in \C\mapsto  e^{2\pi \iu/n} z$ and acting on 
$X= \{z\in \C :|z|<r_0\}$. Then  we can identify $C$ and $X/G$. Moreover, if we again denote by 
$\om$ the metric on $C$ and by $d_0$
 the Euclidean metric or the hyperbolic metric on $X$ (assuming $r_0\le 1$ in the latter case), then under the identification $C\cong X/G$ we have 
 $$\om([z], [w])= \inf\{d_0(u,v): u\in Gz, v\in Gw\}$$
 for  $z,w\in X$, where  we denote the image of $u\in X$ in $C\cong X/G$ by $[u]$. 
 
 Suppose $S$ is a (connected and oriented) surface equipped with
 some path metric $d$. We call $S$ 
{\em locally Euclidean}\index{locally Euclidean} 
if each point $p\in S$ has a neighborhood $U$ isometric to a
Euclidean cone $C$ so that $p$ corresponds to the cone point of
$C$ under the isometry. Then the cone angle $\theta_0>0$ of $C$
is uniquely determined by $p$. The points $p$ with $\theta_0\ne
2\pi$ form a closed and discrete subset of  $S$ and are called the {\em conical
  singularities} or  
{\em cone points}\index{cone!point}\index{cone!point}\index{orbifold!cone point of}\index{conical singularity}\index{singularity!conical}
of the locally Euclidean  surface.

A {\em (simple) Euclidean polygon} is a closed Jordan region $X$  in $\C$ whose boundary consists of finitely many Euclidean line segments. These line segments are called the {\em edges} or {\em sides}, and their endpoints the {\em vertices} or {\em corners} of  
$X$.
 
A {\em (Euclidean) polyhedral surface}\index{polyhedral surface} 
$S$ is a surface obtained by gluing Euclidean polygons together along boundary edges by using isometries. More precisely, we require that $S$ is a locally Euclidean surface that carries a cell decomposition $\DD$ such that each tile $X$ in 
$\DD$   is isometric to a Euclidean polygon $X'$ (with respect to  the induced path metrics on $X$ and $X'$). Note that each  cone 
point of $S$ is necessarily a vertex of $\DD$.

A {\em pillow}\index{pillow} $P$ is a special case of a
polyhedral surface.  It is obtained by gluing two identical
copies $X_{\tt w}$ and $X_{\tt b}$ of a Euclidean polygon $X$
together so that corresponding points on $\partial X_{\tt w}$ and
$\partial X_{\tt b}$ are identified. Then $P$ is a topological
$2$-sphere. The Jordan curve
$ \partial X_{\tt w}=\partial X_{\tt b}\sub P$ is called the 
{\em equator}\index{equator}\index{pillow!equator of} 
of the pillow.

   Often it is useful to consider  $X_{\tt w}\sub P$ as the top face of $P$ colored ``white'' and 
  $X_{\tt b}\sub P$ as the bottom face  colored ``black''. The pillow 
  $P$ carries a natural cell decomposition with $X_{\tt w}$ and 
  $X_{\tt b}$ as tiles, and edges and vertices on the equator 
  $ \partial X_{\tt w}=\partial X_{\tt b}$ of $P$ that correspond to the edges and vertices of $X$. 
  
  The standard orientation on $\C$ induces a natural orientation
  on the polygon $X_{\tt w}$ under a fixed identification
  $X_{\tt w}\cong X$ (represented by a positively-oriented flag,
  for example; see Section~\ref{sec:orient}). We orient $P$ so
  that the induced orientation agrees with this given orientation
  on $X_{\tt w}\sub P$.\index{orientation! of pillow}\index{pillow!orientation of}
  Moreover, we equip $P$ with the unique
  path metric whose restriction agrees with the Euclidean path
  metrics on $X_{\tt w}$ and $X_{\tt b}$.  Then $P$ is a
  polyhedral surface. 
Each of its cone points corresponds 
to a
  vertex of $X$.
 
It is a standard fact that every polyhedral surface $S$ (such as a pillow) admits a natural Riemann surface
 structure (see, for example, \cite[Section~3.3]{Be84}).
One can see this along the following lines. Each point $p\in S$  has a neighborhood $U$ such that there exists an isometry
$\varphi_p\: U\ra D$ with $\varphi_p(p)=0$, where  $D\sub \C$ is a Euclidean disk centered a
 $0$ and equipped with a conformal metric as in \eqref{eq:conemetr1}. By postcomposing $\varphi_p$ with complex conjugation if necessary, we may in addition assume that $\varphi_p$ is 
 orientation-preserving.  One can then show 
 that  these  charts $\varphi_p$, $p\in S$, are holomorphically
 compatible.  So they form a complex atlas ${\mathcal A}$ on $S$. 
 The  analytic equivalence class of  ${\mathcal A}$ is independent of the choices of the isometries $\varphi_p$. Hence $S$ carries a conformal  structure.  If  $S$ is a topological $2$-sphere, then by the uniformization theorem $S$ (as a Riemann surface) is actually conformally equivalent to $\CDach$. 
 
 Let $f\: S\ra S' $ be a continuous map between polyhedral surfaces $S$ and $S'$ equipped with their natural conformal structures. Suppose there exists a set $P\sub S$ that is discrete in $S$ such that the restriction $f|S\setminus P$ is an  orientation-preserving  local similarity, i.e., for each 
 point $p\in  S\setminus P$ there exists a neighborhood $U\sub S\setminus P$ of $p$  such that $f$ maps $U$ onto a neighborhood of $f(p)$ by an orientation-preserving homeomorphism that scales all distances of points in $U$ by a fixed factor $\lambda>0$. Then $f$ is holomorphic. This follows from the definition of the conformal structures on $S$ and $S'$, and the fact that  each point in $P$ is a removable singularity.

Now let   $\mathcal{O}=(\CDach, \alpha)$ be   an  orbifold with the Riemann sphere as the underlying surface  and a ramification function $\alpha\: \CDach \ra \widehat \N$.  We assume that $\OC$ is parabolic or
hyperbolic. Let $\Theta$ be the holomorphic universal orbifold covering map of $\mathcal{O}$ defined on  $X= \C$ in the parabolic and 
 $X=\D$ in the
hyperbolic case (see 
Theorem~\ref{thm:univ_orbi_cover}~\ref{item:ex_univ_orbi2}). 

  In the first case, 
$X=\C$ is equipped with the Euclidean metric, and in the second case $X=\D$ is equipped with the
 hyperbolic metric  as in 
\eqref{eq:def_hyp}. To keep the notation simple,  we denote these metrics both by $d_0$. 
  We will remove the punctures of $\mathcal{O}$ from $\CDach$  
and  consider
\begin{equation*}
  \CDach_0\coloneqq  \CDach \setminus\{p\in \CDach : \alpha(p)=\infty\}.
\end{equation*}
The universal orbifold covering
map  is then a holomorphic branched covering map 
  $\Theta\colon X\to \CDach_0$. The
associated group of deck transformations $G=\pi_1(\OC)$ acts
properly discontinuously on $X$ and each element $g\in G$ is an orientation-preserving isometry on $X$ (see 
Proposition~\ref{prop:prop_deck_trafo}).

The {\em canonical orbifold metric} 
$\om$  on $\CDach_0$ is now defined as\index{canonical orbifold!metric|textbf}\index{metric!canonical orbifold|textbf}\index{o@$\omega$|textbf}\index{orbifold!canonical metric|textbf}
\begin{equation}\label{eq:defcanorbmetr0} 
\om(p,q)= \inf\{ d_0(z,w): z\in \Theta^{-1}(p),\, w\in \Theta^{-1}(q)\}  
\end{equation}
for $p,q\in \CDach_0$.  Note that here the fibers of $\Theta$ are precisely the orbits of $G$. So these are discrete sets in $X$. Since $G$ acts by isometries on   $(X, d_0)$ and
transitively on $\Theta^{-1}(p)$, for each $z_0\in \Theta^{-1}(p)$
there exists $w_0\in \Theta^{-1}(q)$ such that $\om(p,q)=d_0(z_0,
w_0)$. 

This implies that the infimum in \eqref{eq:defcanorbmetr0} is
attained as a minimum and that $\om$ satisfies the triangle inequality. 
It easily follows  that $\om$ is actually a metric on $\CDach_0$.

The uniqueness of the universal orbifold covering map $\Theta$
(Corollary~\ref{cor:unique_univ_orbi} and 
Remark~\ref{rem:holoorbset}) implies  that the metric
$\om$ is uniquely determined by $\OC$ if $\OC$ is hyperbolic and unique up to a
scaling factor if $\OC$ is parabolic. Therefore, in the following we
will refer to $\om$ as the {\em canonical orbifold metric} of $\OC$.

An equivalent way to view  this metric is as follows. The quotient space 
$X/G$ is homeomorphic to $\CDach_0=\Theta(X)$ by  a   homeomorphism $\varphi\colon X/G \to \CDach_0$ 
given  as 
\begin{equation*}
  [z]\mapsto \varphi([z])\coloneqq \Theta(z)
\end{equation*} for $[z]=Gz\in X/G$ 
(see Corollary~\ref{cor:groupquot}). If we define
\begin{equation*}
  \widetilde{\omega}([z],[w])= \inf\{ d_0(g(z), h(w)) : g,h\in G\} 
\end{equation*}
for  $[z],[w]\in X/G$, then $ \widetilde{\omega}$ is a metric on
the quotient $X/G$ 
such that  
the map
$\varphi\colon (X/G,
\widetilde{\omega}) \to (\CDach_0, \omega)$  is an isometry.

Let  $z\in X$ and  $p=\Theta(z)\in \CDach_0$. We consider a small 
ball  $B_z\subset X$ (with respect to $d_0$)
centered  at  $z$. We know  (see the proof of 
Proposition~\ref{prop:prop_deck_trafo}~\ref{item:pi1_O_2}) 
that if $w\ne z$ is close to $z$, then the only points in the orbit $Gw$ close to $z$ are the points in $G_zw$, where $G_z\sub G$ is the stabilizer  subgroup of $z$. So if $B_z$ is small enough, then $B_z/G=B_z/G_z$.

Now $G_z$ is a cyclic group of rotations 
fixing $z$. It follows that $B_z/G=B_z/G_z$ equipped with the metric $\widetilde \om$ is 
a {cone}
\index{cone}
 with cone angle 
 $$2\pi/\#G_z=2\pi/\deg(\Theta,z)=2\pi/\alpha(p). $$ This  cone
is  Euclidean or hyperbolic  depending on whether the orbifold is
parabolic or  hyperbolic. 

If $U_p=\varphi(B_z/G)$, then 
 $(U_p,\om)$  is a neighborhood of  $p$ that is 
 isometric to such a cone. The cone angle  
 $2\pi/\alpha(p)$ 
 of $(U_p,\om)$ is different from $2\pi$  precisely 
  if $2\leq \alpha(p)<\infty$. This is the reason why 
   such points are called the 
 cone points
of the orbifold $\OC=(\CDach, \alpha)$.
Near all other points,
$(\CDach_0, \om)$ is locally isometric to the model space $(X, d_0)$.

We have  $\om(\Theta(z), \Theta(w))\le d_0(z,w)$ for all $z,w\in X$.  A 
stronger condition  is true locally, namely $\Theta$ is a {\em  local radial isometry} in the following sense:   for each $z\in X$ there exists a neighborhood $B_z$ of $z$ in $X$ such that 
\begin{equation}\label{eq:radialisom} 
\om(\Theta(z), \Theta(w))= d_0(z,w)
\end{equation} for all $w\in B_z$. This easily follows from the definition of $\om$ and the fact that for  $w\ne z$ near $z$ the only points 
of $Gw=\Theta^{-1}(\Theta(w))$ near $z$ are the points in $G_zw$ all of which  have the same distance $d_0(z,w)$ to $z$.

 The relation \eqref{eq:radialisom} 
together with a simple covering argument implies that the map $\Theta$
is a {\em path isometry}:\index{path!isometry}\index{visual metric!path isometry} 
if $\ga$ is a path in $X$, then 
\begin{equation}
  \label{eq:om_path_isom}
  \length_\om(\Theta\circ \ga)= \length_{d_0} (\ga).
\end{equation}

The metric space $(\CDach_0, \om)$ is geodesic. Indeed, a geodesic segment joining two points 
$p,q\in \CDach_0$ can be obtained as follows. We can pick points $z\in \Theta^{-1}(p)$ and 
$w\in \Theta^{-1}(q)$  such that 
$\om(\Theta(z), \Theta(w))= d_0(z,w)$. Since $(X,d_0)$ is geodesic, we can find a geodesic segment $\ga$ joining $z$ and $w$. Since $\Theta$ is a path isometry, the path $\Theta\circ \ga$ must then be a geodesic segment  in $(\CDach_0, \om)$ joining $p$ and $q$.

In order to compare the canonical orbifold metric $\omega$ with the chordal
metric $\sigma$,   
fix  $p\in \CDach_0$ and $z\in  \Theta^{-1}(p)$. Then  $\deg(\Theta,z)=\alpha(p)$.  This implies that if $w\in X$ is near $z$ and $q=\Theta(w)$, then 
$$ \sigma(p,q)=\sigma(\Theta(z),\Theta(w))\asymp d_0(z, w)^{\alpha(p)}=
 \omega(p,q)^{\alpha(p)}. $$ 
It follows that there is a neighborhood
$U_p\subset \CDach_0$ of $p$ such that
\begin{equation}
  \label{eq:om_chordal}
  \omega(p,q) \asymp \sigma(p,q)^{1/\alpha(p)},
\end{equation}
for  $q\in U_p$, where $C(\asymp)= C(p)$. 

If  $p\in \CDach$ is a puncture of $\OC$ (i.e.,
if $\alpha(p)=\infty$),  let $\gamma\colon [0,1)\to \CDach_0$ be a path 
that  ends at
$p$ in the sense that $\lim_{t\to 1}\ga(t)=p$. If  $\widetilde{\gamma}$ 
is a lift of $\gamma$ by  $\Theta$ (see Lemma~\ref{lem:liftsofpathsbranched}), then $\widetilde{\gamma}(t)$ 
must leave any compact subset
of $X$ as $t\to 1$. Thus  $\widetilde{\gamma}$ has infinite $d_0$-length which  by  \eqref{eq:om_path_isom} (applied to $\widetilde{\gamma}$) implies that 
 $\gamma=\Theta\circ \widetilde{\gamma}$ has infinite length with respect to $\omega$. In other words,
 for  the metric $\om$ the punctures are infinitely far away from the points in $\CDach_0$. In particular, 
 the metric space $(\CDach_0, \om)$  
is unbounded if $\mathcal{O}$ has punctures.

 If $\mathcal{O}$ has no punctures, then \eqref{eq:om_chordal} and a covering argument implies that 
there exists a constant $C\ge 1$ such that
\begin{equation}\label{eq:compsigom}\frac {1}{C} \sigma(p,q)\le \om(p,q) \le C 
\sigma(p,q)^{\eps}\end{equation}
for $p,q\in \CDach$, where $\eps= \min\{ 1/\alpha(u): u\in \CDach\}$. 
In this case, $\omega$ induces the standard topology on
$\CDach$. The lower bound for $\om$ in
\eqref{eq:compsigom} is also true (on $\CDach_0$) if
$\mathcal{O}$ has punctures (this can  be shown by using the
estimates from Proposition~\ref{prop:conforbsym} below). In particular, the map $\id_{\CDach_0}\: (\CDach_0, \om)\ra (\CDach_0, \sigma)$ is always Lipschitz (but never bi-Lipschitz).  

One can also describe the  metric $\om$ as a singular conformal metric 
with a conformal factor that is smooth everywhere except at the points in 
$\supp(\alpha)$. To see this,  let $z\in X$. We define $\Vert \Theta'(z)\Vert$ as the norm of the derivative $\Theta'(z)$ with respect to the underlying metric $d_0$ on $X$ and the chordal metric $\sigma$ (or rather the spherical metric) on $\CDach$. More explicitly, 
 if $X=\C$, then 
 $$ \Vert \Theta'(z) \Vert= \frac{2|\Theta'(z)|}{1+|\Theta(z)|^2},$$ 
and if $X=\D$, then  
$$ \Vert \Theta'(z) \Vert= \frac{(1-|z|^2) |\Theta'(z)|}{1+|\Theta(z)|^2}.$$ 

 These expressions are essentially special cases of the general formula \eqref{eq:gennorm} and have to be understood as  suitable limits if $\Theta(z)=\infty$. 
 The function $z\mapsto \Vert \Theta'(z)\Vert$ is smooth and positive everywhere on $X\setminus \crit(\Theta)$. 
 
 Since $G=\pi_1(\mathcal{O})$ acts by isometries on $X$, we have 
 $$ \Vert \Theta'(z)\Vert=  \Vert \Theta'(g(z))\Vert$$ for $g\in G$ as follows from the chain rule. So if we set 
 \begin{equation}\label{eq:deforbrho} \rho(p)= \frac1{\Vert \Theta'(z)\Vert}  \end{equation}
  for 
 $p\in \CDach\setminus \supp (\alpha)$ and $z\in \Theta^{-1}(p)=Gz$, then $\rho$ is
 well-defined. Note that $\supp (\alpha)$ includes the punctures and the critical values  of $\Theta$. So on $\CDach\setminus \supp (\alpha)$ the function $\rho$ is smooth and positive. 
 
Now  suppose  $\beta$ is   a  path in $X$ with $\length_{d_0}(\beta)<\infty$ and define 
$\ga=\Theta\circ \beta$. Then  
$$ \length_\sigma(\ga)\lesssim \length_\om(\ga)=\length_{d_0} (\beta)<\infty.$$
If we denote by $ds$ integration with respect to $d_0$-arclength for $\beta$ and by $d\sigma$
integration with respect to $\sigma$-arclength for $\ga$, then    
\begin{align*}\int_{\ga} \rho\, d\sigma&= \int_\beta (\rho \circ \Theta)\Vert \Theta'\Vert\, ds=\int_\beta ds\\
&= \length_{d_0}(\beta) = \length_\om(\ga). 
\end{align*} 
Since $\om$ is a geodesic metric, this and a path lifting argument imply that for all 
$p,q\in \CDach_0$ we have 
$$ \om(p,q) =\inf_{\ga} \int_\ga \rho\, d\sigma, $$
where the infimum is taken over all $\sigma$-rectifiable paths in $\CDach_0$ joining $p$ and $q$. In other words,   $\om$ is the (singular) conformal metric on $\CDach$ with length element $\rho\, d\sigma$.

The local behavior of $\rho$ near  its singularities is described in the following statement.

\begin{prop} \label{prop:conforbsym}
\index{canonical orbifold!metric!density}\index{density of  canonical orbifold metric}\index{o@$\omega$}
Let $\rho$ be the 
conformal density of the canonical orbifold
metric $\om$ 
of a parabolic  or   hyperbolic  orbifold $\mathcal{O}=(\CDach, \alpha)$ as defined in \eqref{eq:deforbrho}. Let 
  $p\in \CDach$. If $\alpha(p)<\infty$, then
 \begin{equation} \label{eq:benesing}\rho(q)\asymp   \frac{1}{\sigma(p,q)^{1-1/\alpha(p)}} 
 \end{equation}  
for $q$ near $p$. 

If $\alpha(p)=\infty$, 
and $\mathcal{O}$ is parabolic, then
\begin{equation} \label{eq:benesing2}\rho(q)\asymp   \frac{1}{\sigma(p,q)}, \end{equation} 
and if $\mathcal{O}$ is hyperbolic, then
\begin{equation} \label{eq:benesing1}\rho(q)\asymp   \frac{1}{\sigma(p,q)\log(1/\sigma(p,q))},  \end{equation} 
 for $q$ near $p$. 

In all these inequalities $C(\asymp)=C(p)$. 
\end{prop} 

If $\alpha(p)=1$, then \eqref{eq:benesing} should be interpreted as $\rho(q)\asymp 1$ for $q$ near $p$. This  corresponds to the fact that $\rho$ is smooth and positive on $\CDach\setminus \supp (\alpha)$.

\begin{proof} As before, we denote by $\Theta\:X \ra \Cdach_0$ the holomorphic universal orbifold covering map of $\mathcal{O}$, where $X=\C$ if $\mathcal{O}$ is parabolic  and $X=\D$ if 
$\mathcal{O}$ is hyperbolic. 

\smallskip
{\em Case 1:}  $\alpha(p)<\infty$. Then there exists $z\in X$ with $\Theta(z)=p$. 
If $w\ne z$ is near $z$ and  $q=\Theta(w)$, then we have 
 $$ \sigma (p,q)\asymp d_0(z,w)^{\alpha(p)} \text{ and } 
 \Vert \Theta'(w)\Vert \asymp d_0(z,w)^{\alpha(p)-1}.$$
Since $\rho(q)=1/\Vert \Theta'(w)\Vert$, inequality  \eqref{eq:benesing} follows.

 The  asymptotics of $\rho$ near a puncture $p\in \CDach$ is much harder to analyze, because $p$ has no preimage in $X$. 
  Without loss of generality we may assume that $p=0$. Then near $p$ chordal and Euclidean metrics are comparable, and so we will state our estimates in terms of the Euclidean metric.

\smallskip
{\em Case 2:} $\alpha(p)=\infty$ and  $\mathcal{O}$ is parabolic.  Then the  signature 
 of $\mathcal{O}$ is $(\infty, \infty)$ or $(2,2,\infty)$ 
 (see the list of parabolic orbifold signatures in 
 Proposition~\ref{prop:parabolicOf}~\ref{item:Of_para2}). 
 
 If the signature is $(\infty, \infty)$, then we may assume that one of the  punctures of $\mathcal{O}$  is  at $p=0$,  the other at $\infty$, and 
 $\Theta(z)=\exp(2\pi\iu z)$.  
For this we match the punctures of the orbifold
with $0$ and $\infty$ by a M\"obius transformation. It changes
the chordal metric only by a factor $\asymp 1$ and so our desired
estimates are not affected. 

Then $\Theta$ maps the upper half-plane 
 $\mathbb{H}$ to the punctured neighborhood  $U_p=\D\setminus\{0\}$ 
 of $p=0$. 
 If  $w\in \mathbb{H}$ and $q=\Theta(w)$, then 
\begin{equation}\label{eq:thth'}
 |\Theta'(w)|\asymp |\exp(2\pi\iu w)| \asymp |\Theta(w)|\asymp  |q|. 
 \end{equation} 
Since $\rho(q)=1/\Vert \Theta'(w)\Vert\asymp 1/|\Theta'(w)|$, inequality \eqref{eq:benesing2} immediately follows. 

If the signature of $\mathcal{O}$ is $(2,2,\infty)$, then we may assume that 
the puncture is at $p=0$, the  two cone points of $\mathcal{O}$ are at $-1$ and $1$, and that 
$$\Theta(z)=1/\cos(2\pi z)= \frac{2  \exp(2\pi\iu z)}{1+ \exp(4\pi\iu z)}. $$ 
Then $\Theta$ maps the half-plane $H=\{ z\in \C:
 \imag (z)>C_0\}$ with $C_0>1$ large to a small punctured  neighborhood $U_p$ of $p=0$. For $w\in H$ and $q=\Theta(w)$ we again have inequalities as in 
 \eqref{eq:thth'}  and  \eqref{eq:benesing2} follows.

 \smallskip
 {\em Case 3:} $\alpha(p)=\infty $  and  $\mathcal{O}$ is hyperbolic. As before, we may assume that $p=0$ and can use the Euclidean rather than the chordal metric near $p$. 
 
This is by far the hardest case. As in Case~2,
 the point $p$ has no
 preimage under  $\Theta$, but in contrast to Case~2 we do not have an explicit expression for $\Theta$.  
To give some intuition how the asymptotics near $p$ arises, we first consider a simple related situation.

 \smallskip
 \emph{Model Case:}  Let  $\Theta_0\colon \Halb \to
 \D\setminus\{0\}$, $u\in \Halb \mapsto \Theta_0(u) \coloneqq  \exp(2\pi \iu u)$. 

 \smallskip 
 If  $q\coloneqq \Theta_0(u)$ for $u\in \Halb$, then $ \abs{\Theta_0'(u)}= 2\pi \abs{q}$ and 
$$\abs{q} = \exp(-2\pi \imag(u)) \text{ or equivalently  } \imag(u)=
   \frac{1}{2\pi}\log(1/\abs{q}).$$  So if 
  we equip $\Halb$ with the hyperbolic metric (given by the  length element \eqref{eq:def_hyponH}) and
 $\D\setminus\{0\}$ with the Euclidean metric,  we obtain
 \begin{equation*}
   \norm{\Theta_0'(u)} =  \imag(u) \abs{\Theta_0'(u)} =
   \abs{q}\log(1/\abs{q}). 
 \end{equation*}
 This means that 
 $1/\norm{\Theta_0'(u)}$ has an asymptotic behavior  similar to
 \eqref{eq:benesing1}, 
where
 $q=\Theta_0(u)\to p=0$.
 
  We will show that for our given universal
 orbifold covering map $\Theta$ we have a relation of  the form $\Theta_0=\Theta\circ \varphi$. We will then derive 
 good distortion bounds for $\varphi$ which will allow us 
 to 
deduce the desired behavior 
of   $\Vert \Theta'\Vert$ 
from the model case.

 Let $U_p=B_\C(0,\delta)\setminus \{0\}$ be a small punctured Euclidean disk around $p=0$. We may assume that $\delta>0$ is so small that $U_p$ does not contain any  point in 
 $\supp (\alpha)$. Let 
 $$ H=\{ z\in \C: \imag (z)>C_0\}$$ with $C_0=\frac1{2\pi}\log(1/\delta)$. 
 Then the  map $$\Theta_0\: H \ra U_p, \ u\in H \mapsto \Theta_0(u)=  \exp(2\pi\iu u)$$ is the universal 
 covering map of $U_p$. 

 Since $U_p\cap\supp (\alpha)=\emptyset$, the 
 map $\Theta$ restricted to any component of $\Theta^{-1}(U_p)$
 is also a covering map over $U_p$. This implies that there
 exists a holomorphic map $\varphi\: H\ra X$ such that
 \begin{equation}
   \label{eq:T0Tphi}
   \Theta_0=\Theta\circ \varphi
 \end{equation}
 on $H$.
 
 
 
We want to show an equivariance property of $\varphi$. For this,  fix $u_0\in H$. Then 
 $q_0\coloneqq \Theta_0(u_0+1)=\Theta_0(u_0)\in U_p$. So if we set $z_0=\varphi(u_0)\in X$ and $z_1=\varphi(u_0+1)\in X$, then 
 $\Theta(z_0)=\Theta(z_1)=q_0$. Hence there exists $g_0\in G=\pi_1(\mathcal{O})$ 
 such that $g_0(z_0)=z_1$. Then both $u\mapsto \varphi (u+1)$ and $g_0\circ \varphi$ are lifts of $\Theta_0$ by  $\Theta$ that send the point $u_0$ to $z_1$. By the uniqueness statement for lifts (Lemma~\ref{lem:liftsofcov}~\ref{item:liuniq}) this implies 
 \begin{equation}\label{eq:autvarphi}
  \varphi(u+1)=g_0(\varphi(u)) \text{ for $u\in H$}. 
  \end{equation}
 
 The map $g_0$ is a 
biholomorphism on $X=\D$ and hence a M\"obius transformation. It can be an elliptic element 
 of finite order (where we allow $g_0=\id_{\D}$), hyperbolic, or parabolic (see \cite[Section~4.3]{Bea} for this standard terminology).

 \smallskip 
{\em Claim.} $g_0$ is parabolic. 

\smallskip
To see this, we consider the cyclic subgroup $G_0$ of $G$ generated by $g_0$ and the quotient $X/G_0$.
 For  all three possible types of $g_0$, the quotient $X/G_0$ carries a natural Riemann surface structure and is conformally equivalent to a bounded region $\Omega\sub \C$. In all cases, an explicit
 biholomorphism $\psi\: X/G_0\ra \Om$ can easily be obtained from a 
 holomorphic branched  covering map $X\ra \Om$ induced by
 $G_0$ (see \cite[pp.~17--19]{Ne53} for a related  argument). Actually, one can think of $\psi$ as a single chart on $X/G_0$ defining the conformal structure on
  $X/G_0$.

 Indeed, if $g_0$ is elliptic, then up to conformal equivalence
 we may assume that $g_0(z)=e^{2\pi \iu/n}z$ with $n\in \N$. So
 $X/G_0=\D/G_0$ is a cone  and conformally equivalent to 
 $\Om=\D$.
 
 If $g_0$ is hyperbolic or parabolic, then, up to conformal equivalence, we may assume that $X$
 is the upper half-plane $\mathbb{H}$ and that 
 $g_0(z)=\lambda z$ with $\lambda>1$  in the hyperbolic and $g_0(z)=z+1$ in the parabolic case. 
 
 So if  $g_0$ is hyperbolic or parabolic, then 
  $X/G_0=
 \mathbb{H}/G_0$ is conformally equivalent to the annulus 
 $\Om=\{z\in \C :1<|z|<\exp(2\pi^2/\log \lambda)\}$
 or
 the punctured unit disk $\Om=\D\setminus\{0\}$, respectively. 
 
 We conclude  that  only in the parabolic case  the region $\Om\cong  X/G_0$ has an isolated boundary point.

To derive the claim from this, we  define a holomorphic map $f\: U_p\ra \Om$ as follows. If $q\in U_p$ is arbitrary, we pick $u\in H$ with $\Theta_0(u)=q$ and set $f(q)=\psi ([\varphi(u)])
\in \Om$.  Here $[z]\in X/G_0$ denotes the orbit of a point $z\in X$ under $G_0$.  
As follows from \eqref{eq:autvarphi}, the  map $f$ is well-defined and holomorphic. 
Since $\Theta=\Theta\circ g$ for all $g\in G_0$, we can also define a unique holomorphic map $\widetilde \Theta\: X/G_0\ra \CDach_0$ by setting $\widetilde \Theta ([z])=\Theta(z)$ for $z\in X$. 

Since $\Om$ is a bounded region, the map $f$ has a removable singularity 
at $p=0$ and hence a holomorphic extension to the disk $D_p=U_p\cup\{p\}$. 
Then $f(0)\in \overline \Om$. Here actually $f(0)\in  \partial \Om$. 
To see this, we pick a sequence $\{u_n\}$ in $H$ with $\text{Im}(u_n)\to +\infty$ as $n\to \infty$. 
Then $\Theta_0(u_n)\to 0$. The sequence $\{ [\varphi(u_n)]\}$ has no limit point in 
$X/G_0$.  Otherwise,  by passing to a subsequence if necessary, we may assume that
 $[\varphi(u_n)]\to [z_0]$, where $z_0\in X$. Then 
 \begin{align*}
  \Theta(z_0)&= \widetilde \Theta([z_0])= \lim_{n\to \infty} \widetilde \Theta([\varphi(u_n)])\\
  &= \lim_{n\to \infty} \Theta(\varphi(u_n)) =\lim_{n\to \infty}  \Theta_0(u_n)=0, 
  \end{align*} 
 contradicting the fact that $p=0$ is a puncture. 
 
 Since the sequence  $\{ [\varphi(u_n)]\}$ has no limit point in 
$X/G_0$ and $\psi$ is a biholomorphism of $X/G_0$ onto $\Om$, we have  
$\psi ([\varphi(u_n)])\to \partial \Om$ as $n\to \infty$.  
It follows  that 
$$ 
f(0)=\lim_{n\to \infty} f(\Theta_0(u_n))=  \lim_{n\to \infty} \psi([\varphi(u_n)])\in \partial \Om. $$ 

Since $f(0)\in \partial \Om$,  the open mapping theorem implies that 
 $f(0)$ must  be an isolated  point on $\partial \Om$.
In the elliptic and hyperbolic case, there are no such points on $\partial \Om$. This shows that $g_0$ is indeed parabolic. We also  see  that for the holomorphic extension of $f$ we have $f(0)=0$.
This finishes the proof of the claim. \smallskip

With the knowledge that $g_0$ is parabolic, we  switch to the  more convenient situation discussed above. Namely, we may assume that $X$ is the upper half-plane 
$\mathbb{H}$ equipped with the hyperbolic metric (given by the length element \eqref{eq:def_hyponH}) and that $g_0(z)=z+1$. This can always be achieved by precomposing 
the original map $\Theta$ with a suitable M\"obius transformation. 

Then in 
\eqref{eq:autvarphi} we have 
$$\varphi(u+1)=\varphi(u)+1$$ for $u\in H$. A 
 biholomorphism $\psi\: X/G_0\ra \Om= \D\setminus \{0\}$ 
is given  by $[z]\mapsto \psi([z])\coloneqq\exp(2\pi \iu z)$.  So it follows from the proof of the claim that there  is a holomorphic function  $f$ on the disk $D_p=U_p\cup\{p\}$ with $f(0)=0$ such that 
\begin{equation}\label{eq:fuvar}
  f(\exp(2\pi\iu u))=f(\Theta_0(u))=\psi( [\varphi(u)])=\exp( 2\pi\iu \varphi(u))  \end{equation} 
for $u\in H$. The function $f$ must be non-constant and so near $0$ it has a Taylor expansion of the form $f(q)=aq^n+\dots$, where
$n\in \N$ and $a\in \C\setminus \{0\}$.
Hence  if $q=\exp(2\pi\iu u)$ is near $0$, or equivalently if $\imag(u)$ is
large,  then 
$$ 
  \exp(-2\pi n \imag(u)) =
  \abs{q}^n  \asymp   
  \abs{f(q)} 
  = \exp (-2\pi \imag(\varphi(u))),   
$$
and so
\begin{equation}\label{eq:imvarphi} 
 \imag(\varphi(u))\asymp \imag(u).  
\end{equation} 
 If we differentiate in \eqref{eq:fuvar} with respect to $u$,  we
 also see that 
 \begin{equation*}
   \abs{q}^n 
   \asymp 
   2\pi \abs{f'(q)} \cdot |q|
   =
   2\pi\abs{\varphi'(u)}\cdot \abs{f(q)}
   \asymp 
   \abs{\varphi'(u)}\cdot \abs{q}^n,  
 \end{equation*}
 and so   
 \begin{equation} \label{eq:phiprimebd} 
  |\varphi'(u)|\asymp 1.
  \end{equation}
Recall from \eqref{eq:T0Tphi} 
that $\Theta_0=
  \Theta\circ \varphi$. Setting $w=\varphi(u)\in X$, we obtain  
  $$\Theta(w)=\Theta(\varphi(u))=\Theta_0(u)=q. $$ Moreover, 
   \eqref{eq:phiprimebd}  and the chain rule immediately give
  $$ \abs{\Theta_0'(u)} = \abs{\Theta'(w)} \cdot |\varphi'(u)| \asymp |\Theta'(w)|. $$

 So if $q=\Theta(w) =\exp(2\pi\iu u)$ is sufficiently close to 
 $0$ (with corresponding $u\in H$), then \eqref{eq:imvarphi}
 shows that 
 \begin{align*}
   \Vert \Theta' (w)\Vert 
   &= 
     \frac {2 \imag(w) |\Theta'(w)|} {1+|\Theta(w)|^2}
     \asymp
  {\imag(w) |\Theta'(w)|}\\  & 
     \asymp   \imag(\varphi(u) ) \abs{\Theta_0'(u)} 
   \asymp \imag(u)  |\Theta_0'(u)|  = \abs{q} \log(1/\abs{q}).
 \end{align*}
 Here the last equality was observed in the model case.  
 Hence 
 \begin{equation*}
   \rho(q)
   = 
   \frac{1}{\Vert \Theta'(w)\Vert}
   \asymp 
   \frac{1}{|q|\log(1/|q|)} 
 \end{equation*}
 for $q$ near $p=0$.  Inequality \eqref{eq:benesing1} follows. 
\end{proof}

We know that  in the absence of punctures the orbifold metric $\om$ is related to the chordal metric $\sigma$ by an inequality
as in \eqref{eq:compsigom}. The following statement further 
clarifies the  relation between these metrics.

\begin{lemma}
  \label{lem:om_chordal}
\index{canonical orbifold!metric}\index{metric!canonical orbifold}\index{o@$\omega$}\index{orbifold!canonical metric}
  Let $\OC=(\CDach, \alpha)$ be a parabolic or hyperbolic  orbifold without punctures. Then
  \begin{enumerate}
     \item
    \label{item:om_chordal_biLip}$(\CDach, \omega)$ and
    $(\CDach, \sigma)$  are bi-Lipschitz equivalent;
    
    \item
 \label{item:om_chordal_qs}
    $\id_{\CDach}\colon  (\CDach, \omega)  \to (\CDach, \sigma)$ is a quasisymmetry. 
   \end{enumerate}
\end{lemma}

It follows from the behavior of the conformal density $\rho$ of $\om$ near the cone points  of $\mathcal{O}$ that the 
bi-Lipschitz equivalence in  \ref{item:om_chordal_biLip} cannot be given by the identity map. 

If  $\OC$ has punctures, then 
$(\CDach_0, \sigma)$ and $(\CDach_0, \omega)$ cannot 
be  
quasisymmetrically equivalent, 
since the first metric space is bounded while the other one is not, and a quasisymmetry preserves boundedness of a space. 

\begin{proof}
 \ref{item:om_chordal_biLip}  We will obtain the desired bi-Lipschitz map $(\CDach, \omega) \ra 
    (\CDach, \sigma)$ from a quasiconformal map on $\CDach$ that behaves like a suitable 
 radial stretch near each point in $\supp (\alpha)$. 
 
 The 
 {\em radial stretch}\index{radial stretch} 
$R_\beta$ for exponent 
  $\beta>0$ is the quasiconformal homeomorphism  $R_\beta \: \C \ra \C$ defined as 
  $$R_\beta(re^{\iu \theta})=r^\beta e^{\iu \theta}$$ 
  for $r\ge 0$, $\theta\in [0,2\pi]$.  The map $R_\beta$ is smooth on $\C\setminus \{0\}$ 
  and we have 
  $$ \Vert DR_\beta(z)\Vert_\sigma \asymp |z|^{\beta-1} $$
  for $z$ near $0$ (see \eqref{eq:sphder2} for the notation used here).  
  Note that $R_\beta$ is the identity on $\partial \D$. This allows us to ``cut and paste''  
  radial stretches together to find a homeomorphism 
  $\varphi \: \CDach \ra \CDach$  such that in a small chordal  disk  $U_p$  centered at a point 
   $p\in \CDach$ with $\alpha(p)\ge 2$, the map $\varphi|U_p$ conjugates to the radial 
   stretch $R_\beta$ on $\D$ with $\beta=1/\alpha(p)$ under a suitable M\"obius transformation that sends $U_p$ onto $\D$ and $p$ to $0$. Moreover, we require that outside these neighborhoods $U_p$ the map $\varphi$ is the identity.   Then $\varphi$ is quasiconformal away from the union of the 
  boundaries $\partial U_p$. Since this union  is a set of   finite Hausdorff $1$-measure and such sets  are removable for quasiconformal maps
 (see \cite[Section~35]{Va}), the homeo\-morphism $\varphi$ 
  is quasiconformal on $\CDach$. By construction 
  we have
   \begin{equation}\label{eq:phispdist}
   \Vert D\varphi(q) \Vert_\sigma  \asymp \sigma(p,q)^{1/\alpha(p)-1}
   \end{equation} 
  for $q$ near $p\in \CDach$ with $\alpha (p)\ge 2$.
  We also have 
  $\Vert D\varphi(q) \Vert_\sigma \asymp 1$ for  almost every $q$  in the complement of a neighborhood of   $\supp (\alpha)$.
   
   In order to establish   that  $\varphi\colon (\CDach, \omega) \to
(\CDach, \sigma)$ is bi-Lipschitz, we will show that it is a map
of \emph{bounded length distortion}. This means that for each 
path $\ga$ in $\CDach$ we have 
$\length_\sigma(\varphi(\gamma)) \asymp \length_\omega(\gamma)$,
where $C(\asymp)$ is independent of $\gamma$. Since
$\omega$ is a  length metric  and $\sigma$ is comparable to a length 
metric (namely the spherical metric) on $\CDach$, this will imply  that
$\varphi$ is bi-Lipschitz as desired. Since the universal
orbifold covering map $\Theta \colon (X, d_0)\to (\CDach,\omega)$ is a
path isometry and every path $\ga$ in $\CDach$ has a lift by  $\Theta$
(see Lemma~\ref{lem:liftsofpathsbranched}), it suffices 
 to show that $\psi \coloneqq \varphi \circ \Theta$ is of bounded
length distortion.

The map $ \psi= \varphi\circ \Theta$ is  quasiregular.  Let $z\in X$ be arbitrary, and  consider a point   $w\in X$ with $w\ne z$  near $z$. If  we set $p\coloneqq \Theta(z)$ and $q\coloneqq \Theta(w)$, then 
  \begin{equation}\label{eq:ThetaDistqr}
    \norm{\Theta'(w)}
    \asymp d_0(z,w)^{\alpha(p)-1} 
    =
    \omega(p,q)^{\alpha(p)-1} 
    \asymp 
    \sigma(p,q)^{1-1/\alpha(p)}
  \end{equation}
  by \eqref{eq:radialisom} and \eqref{eq:om_chordal}. 
  We  denote by 
 $$ \Vert D\psi (w)\Vert\coloneqq \limsup_{w'\to w} \frac{\sigma(\psi(w'), \psi(w))}{d_0(w',w)}$$ the norm of the differential $D\psi(w)$ with respect to 
 the metric $d_0$ on $X$ and the chordal metric $\sigma$ on $\CDach$. If $\alpha(p)\ge 2$, then  \eqref{eq:phispdist} 
and  \eqref{eq:ThetaDistqr} imply  that   
  $$ \Vert D\psi (w)\Vert =  \Vert D\varphi  (q)\Vert_\sigma 
  \cdot \Vert \Theta'(w)\Vert \asymp 1. $$
  This is also true for almost every point $w$  in the complement 
  of a suitable neighborhood of  
  $\Theta^{-1}(\supp (\alpha))$.   
 We conclude  that for  each point $z\in X$ there exists an open  neighborhood $V_z$ of $z$ such 
 that  $\Vert D\psi(w) \Vert \asymp 1$ for almost every  $w\in V_z$. Since $\Vert D\psi(w) \Vert$ is invariant under precomposition with elements of the deck  transformation group of $\Theta$, which acts cocompactly on $X$ (see Proposition~\ref{prop:prop_deck_trafo}~\ref{item:pi1_O_4}), we conclude that $\Vert D\psi(w) \Vert \asymp 1$
 for almost every $w\in X$ with 
$C(\asymp)$ is independent of $w$. 
   
 Since $\psi$ is quasiregular with $\Vert D\psi \Vert \asymp 1$, this   map is of bounded length distortion
 (see \cite[Theorem 2.16] {MV88}), and it follows that   
 $\varphi\: (\CDach, \om )\ra 
  (\CDach, \sigma)$ is indeed a bi-Lipschitz map. 

    \smallskip
  \ref{item:om_chordal_qs}
  Since the homeomorphism  $\varphi$ on $\CDach$ constructed in  \ref{item:om_chordal_biLip}  is quasiconformal, it is a quasisymmetry on $(\CDach, \sigma)$. 
  Hence $\varphi^{-1}\: (\CDach, \sigma) \ra (\CDach, \sigma)$ is
  a quasisymmetry.
  
  This implies that the map $\id_{\CDach}\: (\CDach, \om)\ra (\CDach, \sigma)$ is a quasisymmetry, because it  is the composition of the bi-Lipschitz map 
  $\varphi\: (\CDach, \om)\ra (\CDach, \sigma)$ followed by the quasisymmetry 
  $\varphi^{-1}\: (\CDach, \sigma) \ra (\CDach, \sigma)$. \end{proof}

\index{canonical orbifold!measure|textbf}\index{measure!canonical orbifold|textbf}\index{orbifold!canonical measure|textbf}\index{OAA@$\Omega,\Omega_f$} \index{Lebesgue!measure}\index{measure!Lebesgue}\index{LAA@$\leb$}
 Associated with our orbifold $\mathcal{O}=(\Cdach, \alpha)$  is also  a natural Borel measure 
 $\Om$ on $\CDach$, the 
 {\em canonical orbifold measure}.
To define it, let $\leb_{\CDach}$ be Lebesgue measure (i.e., spherical measure) on $\CDach$. Here (in contrast to 
Chapter~\ref{cha:rati-thurst-maps-1}) we do not impose a normalization 
on $\leb_{\CDach}$
and so $\leb_{\CDach}(\CDach)=4\pi$. 
As before, let $\rho$ be the conformal factor of the orbifold metric $\om$  defined in 
\eqref{eq:deforbrho}. Then for a  Borel set $M\sub \CDach$   we set  
\begin{equation} \label{eq:defOm}
\Om(M)= \int_{M} \rho^2 \, d\leb_{\CDach}.
\end{equation}
We know that $\rho$ is a smooth positive function on $\Cdach\setminus \supp (\alpha)$ and so \eqref{eq:defOm} defines a measure on $\CDach$. Obviously, the measures $\Om$ and 
 $\leb_{\CDach}$ are mutually absolutely continuous. In particular, 
$\Om$ has no atoms even if $\mathcal{O}$ has punctures.  The measure
 $\Om$ is the natural (conformal) measure induced by the
 canonical orbifold metric $\om$ with length element $\rho\,d\sigma$. 
 
 In the hyperbolic case, $\Om$  is independent of the choice of
 $\Theta$ which underlies the definition of $\rho$ and is hence
 unique; 
in the parabolic case, $\Om$ is 
only unique up to a positive multiplicative constant. From the asymptotics of the conformal factor given in 
 Proposition~\ref{prop:conforbsym},  it follows 
 that in the hyperbolic case $\Om$  is always a finite
 measure. We mention without proof that  one
 can show that actually $\Om(\Cdach)=-2\pi \chi(\mathcal{O})$ if
 $\mathcal{O}$ is hyperbolic (this is essentially a special case
 of 
\cite[Theorem~10.4.3]{Bea}).
In the parabolic case  $\Om$    is  finite if and only
 if $\OC$ has no punctures. 
 
Let $\leb_X$ be the natural measure on the orbifold cover $X$, namely 
the  Euclidean area measure (i.e., Lebesgue measure) in the parabolic case when $X=\C$, and the hyperbolic area measure  in the hyperbolic case  when $X=\D$.
  Then 
 one can  consider $\Omega$ as  the ``local'' push-forward of
 $\leb_X$ by the universal orbifold covering map $\Theta$. This is made precise in the following statement. 
\index{push-forward!of measure!by orbifold covering map}

\begin{prop}  \label{lem:prop_cano_orb_meas} 
  The canonical orbifold measure $\Omega$ for a parabolic or hyperbolic  orbifold $\mathcal{O}=(\Cdach, \alpha)$ is the unique Borel measure on $\CDach$ without atoms and with the following property:
  if $M\sub X$ is a Borel set such that the holomorphic universal
  orbifold covering map $\Theta\: X \ra \CDach_0$ of $\mathcal{O}$ is injective on $M$, then 
  \begin{equation}\label{eq:univorbmeas}\leb_X(M)=\Om(\Theta(M)). 
  \end{equation}
  \end{prop}
  We will see in the proof that $\Theta(M)$ is also a Borel set.
  
 This proposition can be reformulated in terms of  Jacobians
 (see Section~\ref{sec:jacobian-mu} for a general discussion of Jacobians):  
  $\Omega$ is the unique measure on $\CDach$ that is absolutely continuous with respect to $\leb_{\Cdach}$ such that for the Jacobian $J_{\Theta, \leb_X, \Om}$ of $\Theta$ with respect to 
$\leb_X$ and $\Om$ we have $J_{\Theta, \leb_X, \Om}= 1$ 
$ \leb_X$-almost everywhere on $X$.

\begin{proof} It is clear that $\Om$ has no atoms, i.e., points $p\in \CDach$ with 
$\Om(\{p\})>0$, because $\Om$ is absolutely continuous with respect to $\leb_{\Cdach}$.

Let $M$ be a Borel set as in the statement.
The map $\Theta$ is a local biholomorphism on $X\setminus\crit(\Theta)$.
So for each point $z\in X\setminus\crit(\Theta)$ there exists a small open ball  
 $B$ centered at  $z$  such that $\Theta$ maps $B$ biholomorphically onto $\Theta(B)$.
 This implies that  $M\sub X$ is a countable disjoint  union of Borel sets each of which is contained in such a ball  $B$ and a countable set $C\sub\crit(\Theta)$. 
 
 It follows  that $\Theta(M)$ is a Borel set. The set $C\sub \crit(\Theta)$ is 
 irrelevant, because $\leb_X(C)=0$ and $\Om(\Theta(C))=0$. So in order to prove 
 \eqref{eq:univorbmeas}, we may assume that $M$ is contained in a ball $B$ on which $\Theta$ 
 is a biholomorphism. 
 Note that by definition of the conformal factor $\rho$ we have 
 $$ \rho(\Theta(z))= \frac1{\Vert \Theta'(z) \Vert }$$ for $z\in B$. 
So  the transformation formula implies 
 \begin{align*} 
 \Om(\Theta(M))&= \int_{\Theta(M)} \rho^2\, d\leb_{\Cdach}=
  \int_{M}(\rho\circ \Theta)^2 \Vert \Theta'\Vert^2\, d\leb_{X}\\
  &=\int_{M} \frac1{\Vert \Theta'(z) \Vert^2 }\Vert \Theta'(z)\Vert^2  \, d\leb_{X}(z)\\
 &=
 \int_M\, d\leb_X=\leb_X(M), 
 \end{align*} 
 and so \eqref{eq:univorbmeas} follows.
 
 With the stated properties the measure $\Om$ is unique. Indeed,  \eqref{eq:univorbmeas}
 uniquely determines $\Om(A)$ for each Borel 
 set $A$ 
 contained in an evenly covered neighborhood of a point $p\in   \CDach\setminus\supp (\alpha)$. If  $B$ is an arbitrary Borel set with  $B\sub  \CDach\setminus\supp (\alpha)$,
 then it  can be represented as a countable disjoint union of such sets $A$ and so $\Om(B)$ is uniquely determined. Finally, since we  have no atoms, $\Om(B)$ is uniquely determined  
 for all Borel sets $B\sub \CDach$. 
\end{proof}

Suppose  $f\: \CDach \ra \CDach$ is a rational Thurston map with
 ramification function $\alpha_f$. We know (see 
 Proposition~\ref{prop:chiO_f_nonneg}) that the orbifold $\mathcal{O}_f=(\CDach, \alpha_f)$ 
  is parabolic or hyperbolic.
 So by the  previous discussion we have  a canonical orbifold
 metric $\om=\om_f$ for $\mathcal{O}_f$ on $\CDach_0$, which we
 will call the 
{\em canonical orbifold metric}\index{canonical orbifold!metric}\index{metric!canonical orbifold}\index{o@$\omega$}\index{orbifold!canonical metric} 
of $f$. Similarly, the orbifold $\mathcal{O}_f$ gives  an associated Borel measure $\Om_f$ on $\CDach$ (as characterized by Proposition~\ref{lem:prop_cano_orb_meas}), called the {\em canonical orbifold measure} of $f$. The metric $\om_f$ and the measure $\Om_f$ are uniquely determined  in the hyperbolic case and unique up to a scaling factor in the parabolic case.  
 
 One of the most important properties of $\om_f$ is that the map $f$ is expanding with respect to this metric if $f$ has no periodic critical points, or equivalently, if $\mathcal{O}_f$ has no punctures (see 
 Proposition~\ref{prop:otherramprops}~\ref{item:rami_infty}).

 \begin{prop}
  \label{lem:R_orbimetric}
\index{canonical orbifold!metric}\index{metric!canonical orbifold}\index{o@$\omega$}\index{orbifold!canonical metric}
  Let $f\colon \CDach \to \CDach$ be a rational Thurston map without
  periodic critical points, and let $\om$ be the canonical orbifold metric of $f$. 
  Then there exists  a constant 
  $\lambda>1$ such that 
 \begin{equation}
  \label{eq:orbi_metric_exp}
  \length_\om(f\circ \gamma) \geq \lambda \length_\om(\gamma) 
\end{equation} for all paths $\ga$ in  $\CDach$.  
 \end{prop}
 
 This lemma is essentially  well known; see  \cite[Theorem~19.6]{Mi} 
or \cite[V.4.3.1]{CG}, for example. As we will see, if $f$  has  a hyperbolic orbifold, then the main idea for the proof is to lift 
inverse branches of $f^{-1}$ by  the universal orbifold covering map 
$\Theta$ to $\D$, and use the fact that 
holomorphic maps of $\D$ into itself are contracting with respect to  the hyperbolic metric. 

 An inequality as in \eqref {eq:orbi_metric_exp}  is actually  true  for arbitrary rational Thurston maps $f$ if we require that  $\ga$ lies  in a sufficiently small neighborhood of the Julia set of $f$.     
 
\begin{proof}   Let $\alpha=\alpha_f$ be the ramification function of $f$. If
  $\OC_f=(\CDach, \alpha)$ is parabolic, then $f$ is a Latt\`es map
  (see Theorem~\ref{thm:Lattesstruc}) and the statement follows from
  Proposition~\ref{prop:paravisorbmet}. So we may assume that
  $\OC_f$ is hyperbolic. 
  Let $\Theta\: X\ra \CDach$ be the holomorphic universal orbifold covering map defined on $X=\D$
  and $\om$ be the canonical orbifold metric of
  $\OC_f$. It is defined on $\CDach$, because $f$ has no periodic critical points and so $\mathcal{O}_f$ has no punctures.

We will show  
that there is a constant $\lambda>1$ such that 
\begin{equation} \label{infiniexp} 
  \Vert f'(q)\Vert_\om\coloneqq  \liminf_{q'\to q} \frac{\om(f(q'), f(q))}{\omega(q',q)} \ge \lambda
\end{equation} 
for all $q\in \CDach$. This expression is the norm 
of the derivative of $f$ with respect to the metric $\omega$.  
In order to establish this inequality, it is enough to show that for each point $q_0\in \CDach$ there exists an open neighborhood $N$ of $q_0$ and a constant 
$\lambda'>1$ such that  $\Vert f'(q)\Vert_\om> \lambda'$ for all $q\in N$. 
Then the estimate \eqref{infiniexp} follows by covering $\CDach$ with  finitely many such sets $N$.

So let $q_0\in \CDach$ be arbitrary, and set $p_0\coloneqq f(q_0)$. We can
find points $z_0,w_0\in \D$ with $\Theta(w_0)=q_0$ and
$\Theta(z_0)=p_0$. Then by Lemma~\ref{lem:lift_Thurst_orbi} we can
find a holomorphic map $A\:\D\ra \D$ such that $A(z_0)=w_0$ and
 \begin{equation}\label{eq:RTfT}
 f\circ  \Theta\circ A=  \Theta. 
 \end{equation}  

So in this setting, we have the following commutative diagram:
 \begin{equation}\label{eq:liftinverseD}
    \xymatrix{
    w_0\in   \D \ar[d]_\Theta & \ar[l]_A z_0\in  \D \ar[d]^\Theta
      \\
     q_0\in  \CDach \ar[r]^{f} & p_0\in \CDach\rlap{.}
    }
  \end{equation}
  One should think of $A$ as a lift of a suitable inverse branch of $f^{-1}$ by 
  the branched covering map $\Theta$.

   By the Schwarz-Pick lemma the derivative of $A$ with respect to the
   hyperbolic metric $d_0$ on $\D$ satisfies  
 $$\Vert A'(z) \Vert: = \frac{(1-|z|^2) |A'(z)|}{1-|A(z)|^2}=
 \lim_{z'\to z} \frac{d_0(A(z'), A(z))}{d_0(z',z)}  \le 1$$ for all
 $z\in \D$. Moreover, here $\Vert A'(z) \Vert=1$ for some point $z\in \D$ if and only if
 $A$ is an automorphism of $\D$. 

 Let us show that in fact $\Vert A'(z) \Vert < 1$ for all $z\in \D$.
 If not, then $A$ is an automorphism of $\D$. Let $u\in \CDach$ be
 arbitrary and $v=f(u)$. Pick  $w\in \Theta^{-1}(u)$ and let $z=A^{-1}(w)$. 
 Then by \eqref{eq:RTfT}  we have 
 $$ \Theta(z)=(f\circ \Theta \circ A)(z)=(f\circ \Theta) (w)=f(u)=v. $$ 
 Since
 $\Theta$ is the universal orbifold covering map of $\mathcal{O}_f$, it follows that 
  \begin{align*}
\alpha(f(u))&=   \alpha(v) = \deg(\Theta,z) = \deg(f\circ \Theta \circ A, z)\\
   & =
   \deg(f, u) \deg(\Theta,w) \underbrace{\deg(A,z)}_{=1}
   =\deg(f,u) \alpha(u).
 \end{align*}
 This implies that $\OC_f$ is parabolic by
 Proposition~\ref{prop:parabolicOf}, which is a
 contradiction.

So in particular, $\Vert A'(z_0) \Vert < 1$. Since the map
$z\mapsto  \Vert A'(z) \Vert$ is continuous,  we can find a
neighborhood $U$ of $z_0$ and a constant $k<1$ such that $\Vert
A'(z) \Vert \le k <1$ for all $z\in U$. The set
$N\coloneqq \Theta(A(U))$ is a neighborhood of $\Theta(A(z_0))=q_0$. If $q\in N$ is arbitrary, then
we can pick a point $w\in U$ with $\Theta(A(w))=q$. Moreover, if $\{q_n\}$ is any sequence contained in $N\setminus\{q\}$  with $q_n\to q$ as $n\to \infty$, then there exists a sequence $\{w_n\}$ in $U$ with $\Theta(A(w_n))=q_n$ for all $n\in \N$ and 
$w_n\to w$ as $n\to \infty$. Hence 
\begin{align*}
\liminf_{n\to \infty}  \frac{\om(f(q_n), f(q))}   {\om(q_n,q)}& =  
\liminf_{n\to \infty}  
\frac{\om((f\circ \Theta \circ A)(w_n), (f\circ \Theta \circ A)(w))}
{\om(\Theta(A(w_n)), \Theta(A(w)))}\\&=
\liminf_{n\to \infty}  
\frac{\om(\Theta(w_n), \Theta(w))}
{\om(\Theta(A(w_n)), \Theta(A(w)))}\\ 
&= \liminf_{n\to \infty}  
\frac{d_0(w_n, w)}
{d_0(A(w_n), A(w))}\\ &=\frac {1}{\Vert A'(w)\Vert}\ge 1/k>1.  
\end{align*}
In the third equality we used the fact that the map $\Theta$ is a local radial isometry (see \eqref{eq:radialisom}).  
We conclude that  $\Vert f'(q)\Vert_\om\ge \lambda'\coloneqq 1/k$ for all $q$ belonging to the neighborhood $N$ of $q_0$. Since $q_0$ was arbitrary,  inequality \eqref{infiniexp} follows. 

Now let  $\ga$ be a path in $\CDach$. Inequality \eqref{infiniexp} in combination with a  covering argument implies that 
$  \length_\om(f\circ \gamma) \geq\lambda \length_\om(\gamma)$.  
So $f$ expands the lengths  of paths with respect to the  metric $\om$ by the factor $\lambda>1$.  
\end{proof}

\end{appendix}





\backmatter

\printindex

\end{document}